\documentclass{memo-l}

\usepackage{amsmath,amssymb,amsthm,bm}
\usepackage{caption,graphicx}
\usepackage[english]{babel}
\usepackage{amscd}
\usepackage{amsgen}
\usepackage[final]{epsfig}
\usepackage{latexsym}
\usepackage{amsfonts}
\usepackage{url}
\usepackage{slashed}
\usepackage[all]{xy}
\usepackage{here}

\usepackage{comment}
\usepackage{fixmath}
\usepackage{float}
\usepackage{scalerel,stackengine}
\usepackage{rotating}

\usepackage{amssymb} 
\usepackage{graphicx} 

\newcommand{\toitself}{\mathbin{\scalebox{.85}{%
    \lefteqn{\scalebox{.5}{$\blacktriangleleft$}}\raisebox{.34ex}{$\supset$}}}}

\usepackage{multicol} 

\stackMath
\newcommand\reallywidehat[1]{%
\savestack{\tmpbox}{\stretchto{%
  \scaleto{%
    \scalerel*[\widthof{\ensuremath{#1}}]{\kern-.6pt\bigwedge\kern-.6pt}%
    {\rule[-\textheight/2]{1ex}{\textheight}}
  }{\textheight}%
}{0.5ex}}%
\stackon[1pt]{#1}{\tmpbox}%
}
\allowdisplaybreaks

\newtheorem{thm}{Theorem}
\newtheorem{lemm}{Lemma}
\newtheorem{prop}{Proposition}
\newtheorem{cor}{Corollary}
\newtheorem{probl}{Problem}
\newtheorem{defn}{Definition}
\newtheorem{rem}{Remark}
\newtheorem{ex}{Example}

\numberwithin{figure}{chapter}

\numberwithin{thm}{chapter}
\numberwithin{lemm}{chapter}
\numberwithin{cor}{chapter}
\numberwithin{prop}{chapter}
\numberwithin{defn}{chapter}
\numberwithin{probl}{chapter}

\numberwithin{rem}{chapter}

\numberwithin{ex}{chapter}

\theoremstyle{definition}

\theoremstyle{remark}

\numberwithin{section}{chapter}
\numberwithin{equation}{chapter}

\makeindex

\begin{document}

\frontmatter

\title{\Huge Braids, conformal module and entropy}


\author{Burglind J\"oricke}
\curraddr{Inst. Math.,  Humboldt-University Berlin,
 Rudower Chaussee 25, 12489 Berlin}

\email{joericke@googlemail.com}

\author{}
\address{}
\curraddr{}
\email{}
\thanks{}

\date{today}




\maketitle

\tableofcontents


\mainmatter

\chapter{Introduction}
\label{intro}

In 1951 J.P. Serre phrased a heuristic principle, widely known as Oka Principle:
{\it In the absense of topological obstructions analytic problems on Stein manifolds admit analytic solutions.} The Cartan-Serre Theorems A and B, and the equivalence between topological and holomorphic classification of principal fiber bundles over Stein manifolds are important examples of this principle.

In \cite{G} M. Gromov offered an interpretation of the Oka Principle
in terms of the $h$-principle for holomorphic functions: ''Oka Principle is an expression of an optimistic expectation with regard to the validity of the h-principle for holomorphic maps in the situation where the source manifold is Stein''.  
The $h$-principle is said to hold for holomorphic mappings from a complex manifold $X$ to a complex manifold $Y$, if each continuous mapping $X\to Y$ is homotopic to a holomorphic one.
Gromov's Oka Principle became the subject of intensive and fruiful research (for a comprehensive account see \cite{Forst}). The emphasis has been on the ''affirmative  side'', i.e. on situations, when the principle holds.

But what about the ''negating side"? 
Is there a genuine interest in looking at situations when Gromov's Oka Principle is violated?\\
After certain hesitation in this respect  several encouraging
observations gave some hope, that
a look at situations, when this important heuristic principle is violated, may
lead to finding interesting structures. The present work is an attempt to start a closer look at this theme.

In more detail, 
a basic problem concerning the ''negating side" of Gromov's Oka Problem is the following.

\smallskip

\noindent {\bf Problem.}
{\it Understand obstructions for Gromov's Oka Principle.}

\smallskip

Obstructions for Gromov's Oka Principle are based
on the relation between conformal invariants of the source and the target, 
when the relation between topological invariants does not completely describe the situation.

Gromov himself referred in his seminal paper \cite{G} from 1989 to mappings from annuli into the twice punctured complex plane as simplest example, where this principle fails and pointed out the special role of mappings from annuli for analyzing the failure of the principle. 
The occurring conformal
invariants of the twice punctured complex plane (called here the conformal modules of conjugacy classes of elements of the fundamental group) are supposed to capture certain "conformal rigidity".

The conformal module of a conjugacy class of elements of the fundamental group $\pi_1(Y,y_0)$ with base point of a complex manifold $Y$ is defined as follows.
The conjugacy class $\hat e$ of an element $e\in \pi_1(Y,y_0)$ 
can be identified with the free homotopy class of curves containing the curves of the class $e$.
A continuous map $g$ from an
annulus $A = \{ z \in {\mathbb C} : r < \vert z \vert < R \}$ into
${\mathcal X}$ is said to represent the conjugacy class $\hat e$
if for some (and hence for any) $\rho \in (r,R)$ the map $g :
\{\vert z \vert = \rho \} \to {\mathcal X}$ represents $\hat e$.

The definition of the conformal module of $\hat e$ is based on Ahlfors' concept of conformal module and extremal length. According to Ahlfors the conformal module of an annulus $\,A= \{z \in \mathbb{C}:\; r<|z|<R\}\,\,$ is equal to
$\,m(A)= \frac{1}{2\pi}\,\log(\frac{R}{r})\,,$ and its extremal length equals $\lambda(A)=\frac{1}{m(A)}$.
For any complex manifold $\mathcal{X}$ and any conjugacy class $\hat e$ of elements of the fundamental group of $\mathcal{X}$ the conformal module $\mathcal{M}(\hat{e})$ of $\hat e$ is defined as the supremum of the conformal modules of annuli that admit a holomorphic mapping into $\mathcal{X}$ representing $\hat e$. The extremal length $\Lambda(\hat{e})$ equals $\frac{1}{\mathcal{M}(\hat{e})}$.

An obstruction to Gromov's Oka principle is given by the following observation.\\
{\it For a  holomorphic mapping $f:X\to Y$ from a complex manifold $X$ to a complex manifold $Y$ and any element $e \in \pi_1(X,x_0)$ the conformal module of the free homotopy class of its image $\widehat{f_*({e})}$ is  not smaller than the conformal module of the free homotopy class $\hat{e}$ of the element itself.}

It was an encouringing observation, that for the twice punctured complex plane it is possible to give good estimates of the conformal invariants proposed by Gromov.
More precisely, in Chapter \ref{chapter3-braids} we give upper and lower bounds for the conformal  module of conjugacy classes of elements of the fundamental group of the twice punctured complex plane. The upper and lower bound differ by multiplicative constants not depending on the conjugacy class.

Interestingly, the notion of  the conformal module of conjugacy classes of elements of the fundamental group appeared already 1974 (without name) much before Gromov's work.
It was used in the paper \cite{GL} which was motivated
by the interest of the authors in Hilbert's Thirteen's Problem for
algebraic functions. In Chapter \ref{chapter8} we  present some related historical remarks. 
The complex manifold in this case was the space
$\mathfrak{P}_n$, the space of
monic polynomials of degree $n$ without multiple zeros. (The adjective ''monic'' refers to the property that the coefficient by the highest power of the variable equals $1$).
The elements of the fundamental group of this space are braids on $n$ strands. Thus, the invariant that was considered in  \cite{GL} is the conformal module of conjugacy classes of braids. It was applied to the question of global reducibility of families of polynomials  $\mathfrak{P}_n$ that depend holomorphically on a parameter varying over a connected open (i.e. non-compact) Riemann surface (in particular over an annulus).

Braids play a role in several mathematical fields. It is therefore
not surprising that invariants of conjugacy classes of braids occurred in different areas of
mathematics. There is a dynamical invariant, the entropy. The dynamical concept was considered in 1979 in \cite{FLP} in the Asterisque volume dedicated to Thurston's theory of surface homeomorphisms. Thurston's theory of surface homeomorphisms was motivated by his celebrated geometrization conjecture. For the dynamical concept one assigns to each braid on $n$ strands an isotopy class of self-homeomorphisms of the unit disc all elements of which  fix the unit circle pointwise and permute a set $E_n$ of $n$ chosen points in the unit disc. The entropy of the braid is the infimum of the entropy of the mappings in the class. The entropy does not change under conjugation, in other words it is an invariant of conjugacy classes of braids.
The entropy of surface homeomorphisms has received a lot of
attention.

It turns out that the two invariants related to the dynamical aspect and to the conformal aspect of braids are related. The following theorem holds.
\medskip

\noindent {\bf Main Theorem.} {\it For each conjugacy class of braids
$\hat b \in \hat{\mathcal B}_n$, $n \geq 2$, the following equality
holds
$$
{\mathcal M} (\hat b) = \frac\pi2 \, \frac1{h(\hat b)} \, .
$$
}
Here $h(\hat b)$ denotes the entropy of the conjugacy class $\hat b$.
\medskip

The relation between the invariants coming from different fields has applications. 
First, results related to the more intensively studied dynamical invariant entropy may be applied
to problems that use the concept of conformal module. For instance, Chapter \ref{chapter8}
contains a short conceptional proof (of a slightly stronger version) of the result of \cite{GL},
based on a result \cite{P} related to entropy.

Vice versa, methods of complex analysis including quasi-conformal mappings, that can be used for the study of the conformal module, may give results related to the entropy.
For example, the relation between the conformal and the dynamical invariants together with the effective estimates of the conformal module give 
effective upper and lower bounds for the entropy of pure $3$-braids.

There is an invariant of braids (not merely an invariant of conjugacy classes of braids), 
the conformal module (or extremal length) of braids with totally real horizontal boundary values. It uses Ahlfors' extremal length of rectangles instead of the extremal length of annuli. This notion is more powerful for applications to problems related to Gromov's Oka Principle than the extremal length of conjugacy classes of braids. For instance, it is used to obtain
quantitative statements concerning Gromov's Oka Principle in the spirit of analogs of the Geometric Shafarevich Conjecture (see Chapter \ref{chapterfin}).

We tried to make the work self-contained and coherent for graduate students.
Chapter \ref{chapter2} provides prerequisites from several topics, like Riemann surfaces, braids and mapping classes, as well as Teichm\"uller theory.
In some cases proofs are included, for instance,
in the case when the material is less standard and not extensively treated in the literature, or when the idea behind a statement will occur to be important in the following.
Readers who are familiar with the material may skip this chapter and
consult it later for notation and for some less standard facts
that are included in this chapter.

In Chapter \ref{chapter-entropy} a self-contained thorough treatment on the entropy of braids and mapping classes is given. One of the theorems of  \cite{FLP} on irreducible mapping classes on closed surfaces is proved along the lines of \cite{FLP},  another theorem is proved in a different way. We include detailed prerequisites
on foliations generated by quadratic differentials. A thorough treatment of the
entropy of mapping classes on finite Riemann surfaces of second kind (i.e. on connected Riemann surfaces that are neither closed nor punctured and have finitely generated  fundamental group) is presented. This is needed to treat the entropy of braids and the entropy of reducible mapping classes.

The proof of the Main Theorem for irreducible braids is given in Chapter \ref{chapter4}.
The chapter starts with a short introduction to Ahlfors' concept of the extremal length of families of curves in planar domains, and presents an extension of the concept to homotopy classes of curves in complex manifolds. The chapter contains
explicit computations of the extremal length of some examples
of conjugacy classes of $3$-braids (which means by the Main Theorem, to give another way of computing some entropies), and it contains
explicit computations
of the extremal length with totally real horizontal boundary values of some examples of $3$-braids.

In Chapter \ref{chapter6} we define the irreducible braid components of a conjugacy class of reducible pure braids and show that the conformal module of the conjugacy class of a pure braid equals the minimum of the conformal modules of the irreducible components. Similarly, we define the irreducible nodal components of the conjugacy class of mapping classes associated to the braid, and show that the entropy of the mapping class is equal to the maximum of the entropies of the irreducible nodal components. The Main Theorem for irreducible pure braids implies the Main Theorem for reducible pure braids.
The conjugacy class of a reducible pure braid can be recovered from its irreducible components. The conjugacy class of the mapping class can be recovered up to products of commuting Dehn twists from the conjugacy classes of the irreducible nodal components.

Similar statements for not necessarily pure braids are contained in Chapter \ref{chapter7a}.
As mentioned above Chapter \ref{chapter8} discusses global reducibility of holomorphic families of
polynomials.

The deepest applications of the concept of the conformal module of braids or their conjugacy classes (more generally, of elements of fundamental groups or their conjugacy classes) are contained in Chapters \ref{chapterGrom} - \ref{chapter3-braids}.

In Chapter \ref{chapterGrom} we study situations when Gromov's Oka Principle fails, using the concept of the conformal modules of  conjugacy classes of elements of the fundamental group. 
For instance, we call a mapping $f:X\to Y$ from a connected finite open Riemann surface $X$ to a complex manifold $Y$ a Gromov-Oka mapping, if for each conformal structure on $X$ with only thick ends the mapping is homotopic to a holomorphic mapping.
 (A conformal structure on a connected finite open Riemann has only thick ends if it makes the manifold a closed Riemann surface with a finite number of closed discs but no point removed.)
The Gromov-Oka mappings from finite open Riemann surfaces to the twice punctured Riemann surface are completely described.

In Chapter  \ref{chapterEl} we 
prove theorems related to the failure of the Gromov-Oka Principle for bundles whose
fibers are Riemann surfaces of type $(1,1)$ (in other words, the fibers are once punctured tori). The problem can be reduced to the respective problem for
$(0,4)$-bundles (and, hence, for mappings into  $\mathfrak{P}_3$) by considering double branched coverings.

Chapter \ref{chapter3-braids} is devoted to effective upper and
lower bounds for the extremal length with totally real horizontal boundary values of any element of the fundamental group of the twice punctured complex plane  (not merely of the extremal length of conjugacy classes). The bounds differ by a multiplicative constant not depending on the element.
This is done in terms of a natural syllable decomposition
of the word representing the element. In the last section (Section \ref{sec:3-braids5})
we give respective estimates for $3$-braids. The extremal length with totally real horizontal boundary values
is capable to give more subtle information regarding Gromov's Oka Principle and limitations of its validity than the respective invariant of conjugacy classes.
Moreover, the estimates for the extremal length with totally real horizontal boundary values of pure braids imply estimates of the extremal length of conjugacy classes of these braids, hence, of their entropy.

In Chapter \ref{Ch9} we give effective estimates for the growth of the number of those  elements of the braid group modulo center $\mathcal{B}_3\diagup\mathcal{Z}_3$, 
whose extremal length with totally real boundary values is positive and does not  exceed a positive number $Y$. 
As a corollary we give an alternative proof of the result of Veech \cite{Ve} on the exponential growth of the number of conjugacy classes of elements of $\mathcal{B}_3\diagup\mathcal{Z}_3$ that have positive entropy not exceeding  $Y$. The proof does not use deep techniques from Teichm\"uller theory.

In Chapter \ref{chapterfin} we apply the concept of conformal module and extremal length to obtain quantitative statements on the limitation for Gromov's Oka Principle. More specifically, for any finite open Riemann surface $X$ (maybe, of second kind) we
give an effective upper bound for the number of irreducible holomorphic mappings up to homotopy from $X$ to the twice punctured complex plane, and an effective upper bound for the number of irreducible holomorphic torus bundles up to isotopy on such a Riemann surface.
These statements are analogs of the Geometric Shafarevich Conjecture and the Theorem of de Franchis, that state the finiteness of the number of certain holomorphic objects on closed or punctured Riemann surfaces.

\noindent \textit{Acknowledgement}. While working on this project the author had the opportunity to gain from the wonderful
working conditions at IHES, Max-Planck-Institute for Mathematics,
Weizmann Institute,  CRM Barcelona, and  at  Humboldt-University Berlin,
in particular at the SFB "Raum-Zeit-Materie".
Other parts of the work were done while visiting the \'{E}cole Normale Sup\'{e}rieure Paris, the Universities of
Grenoble, Bern, Calais, Toulouse, Lille, Bloomington, and Canberra. 
Talks, or series of lectures, given on
the topic of the paper at various places encouraged research and
helped to improve the explanation. A talk given at the Algebraic
Geometry seminar at Courant Institute and a stimulating discussion
with F.Bogomolov and M.Gromov had a special impact. The author is
indebted to many other mathematicians for interesting and fruitful
discussions and for information concerning references to the
literature. Among them are B. Berndtsson, R.Bryant, D.Calegari, F.Dahmani, P.Eyssidieux,
B.Farb, P.Kirk, M.Reimann, S.Orevkov, L.Stout, O.Viro and
M.Zaidenberg. O.Viro asked about a notion of conformal module
of braids rather than of conjugacy classes of braids. M.Reimann gave
references to the literature related to computation of the conformal
module of special doubly connected domains and of quadrilaterals.
The author is grateful to B.Farb who suggested to use the concept of conformal
module and extremal length for a proof of finiteness theorems, and to B.Berndtsson
for proposing the kernel for solving the $\bar{\partial}$-problem that arises in the
proof of Proposition \ref{propfin1a}.

Special thanks go to Alexander Wei{\ss}e for teaching how to draw
figures and for producing  the essential parts of some
difficult figures.
The author is much obliged to F.Dufour and 
Marie-Claude Vergne for drawing some other figures, and to C\'ecile Gourgues for typing a preliminary version of some chapters.

\chapter{Riemann Surfaces, Braids, Mapping Classes, and Teichm\"uller Theory.}
\label{chapter2}
\setcounter{equation}{0}

In this chapter we provide prerequisites from several topics, like Riemann surfaces, braids and mapping classes, as well as Teichm\"uller theory.
For completeness
and convenience of the reader and for fixing notation and
terminology we give here a recollection of basic facts, that will be used later
frequently. References to the literature are given. In some cases proofs are included, for instance,
in the case when the material is less standard and not extensively treated in the literature, or when the idea behind a statement will occur to be important in the following.
Readers who are familiar with the material may skip this chapter and
consult it later for notation and for some less standard facts
that are include in this chapter.

\section{Free homotopy classes of mappings and conjugacy classes of group elements}
\label{sec:2.3b}

\noindent {\bf The change of the base point.}
Let $\mathcal{X}$ be a connected smooth manifold with base point $x_0$ and non-trivial fundamental group $\pi_1(\mathcal{X},x_0)$.
\index{group ! fundamental group} \index{$\pi_1(X,x_0)$}  Let $\alpha$ be an arc in $\mathcal{X}$ with initial point $x_0$ and terminal point $x$. Change the base point $x_0 \in \mathcal{X}$
along a curve $\alpha$ to the point $x \in \mathcal{X}$. This leads to an isomorphism $\mbox{Is}_{\alpha}: \pi_1(\mathcal{X},x_0) \to \pi_1(\mathcal{X},x)$ of fundamental groups induced by the correspondence $\gamma \to \alpha^{-1} \gamma \alpha$ for any loop $\gamma$ with base point $x_0$ and the arc $\alpha$ with initial point $x_0$ and terminal point $x$. We will denote the correspondence $\gamma \to \alpha^{-1} \gamma \alpha$ between curves also by $\mbox{Is}_{\alpha}$.
\index{${\rm Is}_{\alpha}$}

We call two homomorphisms $h_j:G_1 \to G_2, \, j=1,2,\,$ from a group $G_1$ to a group $G_2$ conjugate if there is an element $g' \in G_2$ such that for each $g \in G_1$ the equality $h_2(g)= {g'}^{-1} h_1(g) g'$ holds.
For two
arcs $\alpha_1$ and $\alpha_2$ with initial point $x_0$ and terminal point $x$ we have
$\alpha_2^{-1} \gamma \alpha_2= (\alpha_1^{-1}\alpha_2)^{-1} \alpha_1^{-1} \gamma \alpha_1 (\alpha_1^{-1}\alpha_2)$.
Hence, the two isomorphisms $\mbox{Is}_{\alpha_1}$ and $\mbox{Is}_{\alpha_2}$ differ by conjugation with the element of $\pi_1(\mathcal{X},x)$ represented by $\alpha_1^{-1}\alpha_2$.

Free homotopic curves are related by homotopy with fixed base point and an application of a homomorphism $\mbox{Is}_{\alpha}$
that is defined up to conjugation. Hence, free homotopy classes of curves can be identified with conjugacy classes of elements of the fundamental group  $\pi_1(\mathcal{X},x_0)$ of $\mathcal{X}$.

Let $\mathcal{X}$ and $\mathcal{Y}$ be oriented smooth connected non-compact  surfaces with finitely generated fundamental group, with base points $x_0 \in \mathcal{X}$ and $y_0 \in \mathcal{Y}$. For a continuous mapping $F:\mathcal{X} \to \mathcal{Y}$ with $F(x_0)=y_0$ we denote by $F_*: \pi_1(\mathcal{X},x_0) \to \pi_1(\mathcal{Y},y_0) $ the induced map on fundamental groups. For each element $e_0\in \pi_1(\mathcal{X},x_0)$ the image $F_*(e_0)$ is called the monodromy
\index{monodromy} along $e_0$, and the homomorphism $F_*$ is called the monodromy homomorphism corresponding to $F$. \index{$F_*$}
The mapping $F\to F_*$ defines a one to one correspondence from
the set of homotopy classes of continuous mappings $\mathcal{X}\to \mathcal{Y}$ with fixed base point $x_0$ in the source and fixed value $y_0$ at the base point to the set of homomorphisms $\pi_1(\mathcal{X},x_0) \to \pi_1(\mathcal{Y},y_0) $.

Consider a free homotopy $F_t, \, t \in (0,1)$, of continuous mappings from $\mathcal{X}$ to $\mathcal{Y}$. Consider the curve $\alpha(t)\stackrel{def}=F_t(x_0),\, t\in [0,1]$. Suppose $F_0(x_0)=y_0$ and $F_1(x_0)=y_1$. Each curve $\beta$ in $\mathcal{Y}$ with initial point $y_0$ and terminal point $y_1$ defines an isomorphism  $\pi_1(\mathcal{Y},y_0) \to \pi_1(\mathcal{Y},y_1) $ by considering the compositions of curves $\beta^{-1} \gamma \beta$ for loops $\gamma$ representing elements of $\pi_1(\mathcal{Y},y_0)$. For different curves $\beta$ the isomorphisms differ by conjugation. Identifying fundamental groups with different base point by a chosen isomorphism of the described kind
we obtain a mapping that assigns to each free homotopy class of maps  $\mathcal{X}\to \mathcal{Y}$ a conjugacy class of homomorphisms $\pi_1(\mathcal{X},x_0) \to \pi_1(\mathcal{Y},y_0) $.
The mapping is surjective. It is also injective by the following reason. Let $F_j:\mathcal{X}\to \mathcal{Y},\,j=0,1,$ be continuous mappings, $F_j(x_0)=y_0$, such that for an element $e\in \pi_1(\mathcal{Y},y_0)$ the equality $(F_1)_*=e^{-1}(F_0)_*\,e$ holds. Let $\alpha$ be a smooth curve in $\mathcal{Y}$ that represents $e$. There exists a free homotopy $F^t,\, t\in [0,1],$ of mappings such that $F^0=F_0$ and $F^t(x_0)=\alpha(t),\, t\in[0,1]$. Then $(F^1)_*=e^{-1}(F^0)_*\,e$. Hence, $F_1$ and $F^1$ define the same homomorphism $\pi_1(\mathcal{X},x_0) \to \pi_1(\mathcal{Y},y_0) $, and hence, they are homotopic.
We obtain the following theorem (see also \cite{Ha},\cite{St}, \cite{Sp}.)

\begin{thm}\label{thmEl-1}
The free homotopy classes of continuous mappings from $\mathcal{X}$ to $\mathcal{Y}$
are in one-to-one correspondence to the set of conjugacy classes of homomorphisms between the fundamental groups of $\mathcal{X}$ and $\mathcal{Y}$.
\end{thm}

\section{Riemann surfaces}
\label{sec:2.0}
We will use the common definition of
a surface, of a smooth surface, of a surface with boundary, and of a Riemann surface.
A homeomorphic mapping $\omega:S\to X$ from a surface $S$ onto a Riemann surface $X$ is called a conformal structure or a complex structure, having in mind, that the pull back of a complex analytic atlas on $X$ provides a complex analytic atlas on $S$.

 If the interior of a smooth surface with boundary
is equipped with a complex structure, the surface with boundary  is called (following Ahlfors) a bordered Riemann surface.\index{Riemann surface ! bordered} \index{surface}  \index{surface ! Riemann surface} \index{surface ! with boundary}

A surface (not a surface with boundary) is called open  if no connected component is compact. A compact surface (without boundary) is called a closed surface.
A closed surface with finitely many points removed is called a punctured surface. The removed points are called punctures.
A connected closed oriented surface or a connected closed Riemann surface
is called of type $(g,m)$,
if it has genus ${g}$ and is equipped with ${{m}}$ distinguished points.  \index{surface ! punctured}  \index{surface ! open}\index{surface ! closed}
\index{points ! distinguished points} \index{surface ! Riemann surface of type $(g,m)$}

The intersection number  of an ordered pair of smooth loops with transversal  intersection in a smooth oriented surface $X$  is the sum of the intersection numbers over all intersection points. The intersection number at an intersection point equals  $+ 1$ if the orientation determined by the tangent vector to the first curve
followed by the tangent vector to the second curve is the orientation of $X$, and equals $-1$ otherwise. The intersection number does not change under homotopy, i.e. the intersection number of a homotopic pair of smooth curves with transversal intersection is equal to the intersection number of the original pair.

If the loops are not smooth or do not intersect transversally, their intersection number is defined as intersection number of a homotopic pair of smooth loops with transversal intersection.
\index{intersection number}

We will consider only oriented surfaces, often without further mentioning. If not mentioned otherwise, the considered surfaces will be connected.
A surface is called finite if its fundamental group is finitely generated. A finite surface is obtained from a closed surface by removing finitely many disjoint simply connected closed sets (called holes).
The fundamental group of a surface of genus $g$ with $m>0$ holes is a free group in $2g+m-1$ generators. \index{group ! fundamental group}
\index{surface ! finite} \index{holes}

Each finite connected open Riemann surface $X$ is conformally equivalent
to a domain (denoted again by $X$) on a closed Riemann surface $X^c$ such that each connected component of the complement $X^c \setminus X$ is either a point or a closed topological disc with smooth boundary \cite{Sto}.

A connected Riemann surface is said to be of first kind, if it is a closed or a punctured Riemann surface, otherwise it is said to be of second kind. If all holes of a finite open Riemann surface are closed topological discs, the Riemann surface is said to have only thick ends. \index{Riemann surface ! of first kind} \index{Riemann surface ! of second kind} \index{Riemann surface ! with only thick ends}

\smallskip

\noindent {\bf Standard bouquets of circles.} Let $X$ be an oriented smooth open surface of genus $g\geq 0$ with $m>0$ holes, equipped with a
base point $q_0$.
The union $B$ of non-contractible circles in $X$ with base point $q_0$ is called a bouquet of circles in $X$,
\index{bouquet of circles ! in $X$}
if $q_0$ is the only common point of any pair of circles in $B$, and different circles in $B$ with base point $q_0$ represent different elements of $\pi_1(X,q_0)$.
Suppose $B$ is the union
of simple closed oriented curves $\alpha_j,\,\beta_j$, $ j=1,\ldots ,g',$  and  $\gamma_{k}$, $k=1,\ldots,m'$, with base point $q_0$ with the following property.
Labeling the rays of the loops emerging from the base point $q_0$ by $\alpha_j^-,\,\beta_j^-$  $\gamma_j^-$, and the incoming rays by $\alpha_j^+,\,\beta_j^+$  $\gamma_j^+$, and
moving in counterclockwise direction along a small circle around $q_0$ we meet the rays in the order
\begin{align*}
\ldots, \alpha^-_j,\beta^-_j,\alpha^+_j,\beta^+_j,\ldots, \gamma^-_k,\gamma^+_k,\ldots\;.
\end{align*}
Then $B$ is called a standard bouquet of circles in $X$.

Let  $B$ be a standard bouquet of circles in $X$ with base point $q_0$. \index{bouquet of circles ! standard bouquet of circles in $X$}
If the collection $\mathcal{E}$ of elements of the fundamental group $\pi_1(X,q_0)$ represented by the collection of curves in $B$ is a system of generators of $\pi_1(X,q_0)$ (then in particular, $g'=g$, $m'=m$), we call $B$ a standard bouquet of circles for $X$, and say that the system $\mathcal{E }$ of generators
 is associated to a standard bouquet of circles for $X$. \index{bouquet of circles ! standard bouquet of circles for $X$} \index{generators ! associated to a standard bouquet of circles for $X$} \index{$\mathcal{E}$} \index{$B$}
See also Figure \ref{fig8.3}.
\begin{figure}[H]
\begin{center}
\includegraphics[width=8.5cm]{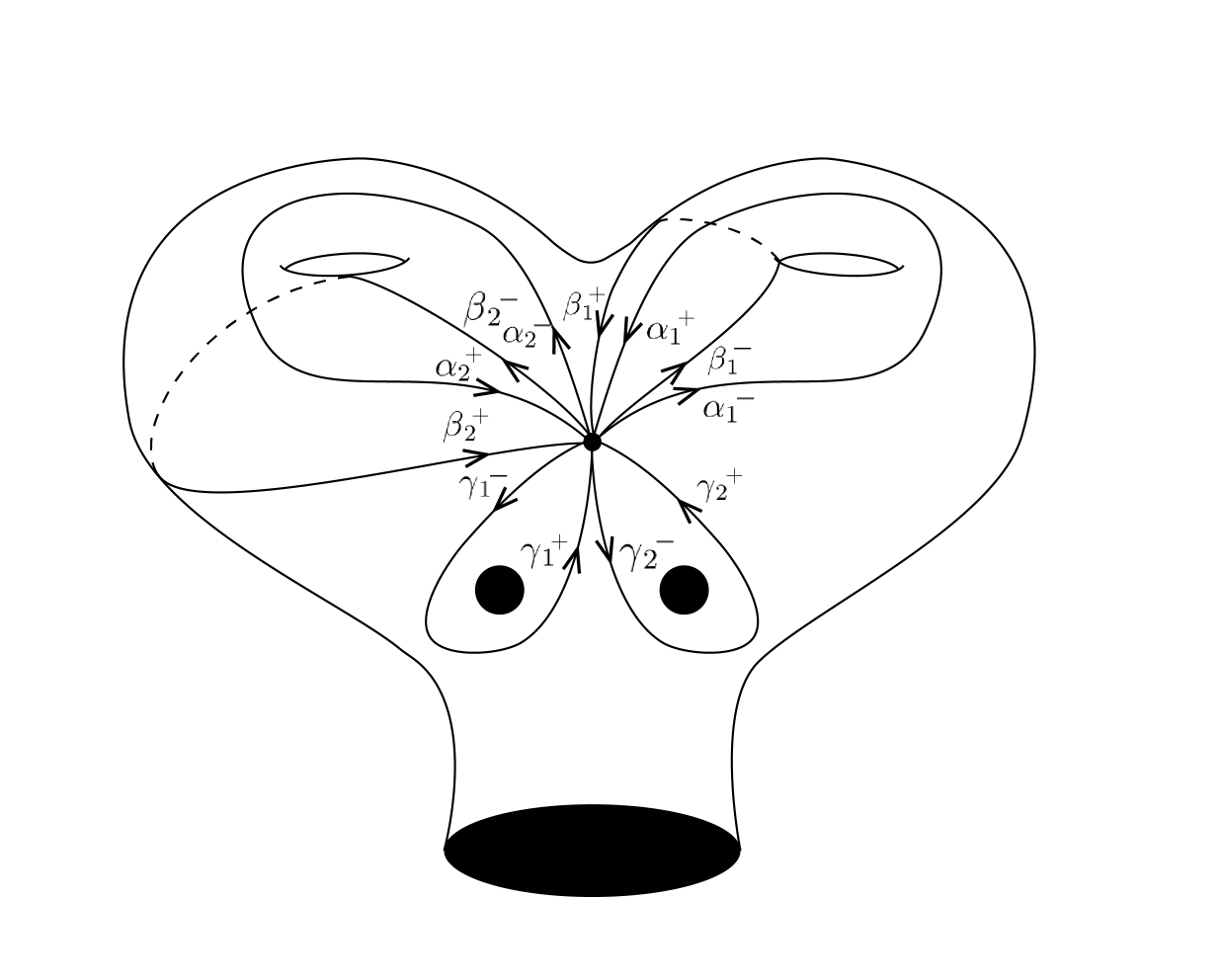}
\end{center}
\caption{A standard bouquet of circles for a connected finite open Riemann surface}\label{fig8.3}
\end{figure}

Label the generators $\mathcal{E}\subset  \pi_1(X,q_0)$ associated to a standard bouquet of circles for $X$ as follows.
The elements $e_{2j-1,0}\in \pi_1(X,q_0),$  $j=1,\ldots ,g,$  are represented by $\alpha_j$, the elements $e_{2j,0}\in \pi_1(X,q_0) , \, j=1,\ldots ,g,$  are represented by $\beta_j$, and the elements $e_{2g+k,0}\in \pi_1(X,q_0), \, k=1,\ldots ,m,$  of $\pi_1(X,q_0)$ are represented by $\gamma_k$. The system $\mathcal{E}=\{e_j\}_{j=1}^{2g+m-1}$labeled in this way is called a standard system of generators of $\pi_1(X,q_0)$. \index{generators ! standard system of generators for $X$}

A standard bouquet of circles for a connected finite open Riemann surface can be obtained as follows. Recall that the Riemann surface $X$ is conformally equivalent to a domain in a closed Riemann surface $X^c$ of genus $g$.
Take a fundamental polygon of the compact Riemann surface $X^c$, so that the projection of its sides to $X^c$ do not meet the holes of $X$. The projections of the sides provide a collection of simple closed curves $\alpha_j$, $\beta_j$, $j=1,\ldots,g,$ with base point $q_0$, that are pairwise disjoint outside  $q_0$, and represent the elements of the fundamental group $\pi_1(X^c,q_0)$. Each pair $(\alpha_j, \beta_j)$ corresponds to a handle of $X^c$.
Cut $X^c$ along the curves $\alpha_j$, $\beta_j$, $j=1,\ldots,g$. We obtain a simply connected domain.
Denote the connected components of $X^c \setminus X$ by $\mathcal{C}_1, \mathcal{C}_2,\ldots, \mathcal{C}_{m}$. (The $\mathcal{C}_{\ell}$ are also the connected components of $X^c \setminus \big(X\bigcup_j (\alpha_j\cup \beta_j)\big) $.)
 Each component is either a point or a closed disc. We associate
to the components $\mathcal{C}_\ell,\, \ell=1,\ldots,m-1,$ a collection of $m-1$
simple closed curves $\gamma_{\ell}$ in $X$ with base point $q_0$ that meet the union $\bigcup_j (\alpha_j\cup \beta_j)$ exactly at the point $q_0$ and are pairwise disjoint outside $q_0$.
Moreover, each $\gamma_{\ell}$ is contractible in $X\cup \mathcal{C}_\ell$ and divides $X$ into two connected
components, one of them containing $\mathcal{C}_\ell$. We will say that $\gamma_{\ell}$ surrounds $\mathcal{C}_\ell$. We orient $\gamma_{\ell}$ so that $\mathcal{C}_{\ell}$ is on the left when walking along the curve $\gamma_{\ell}$ equipped with this orientation.
We may choose the $\gamma_{\ell}$ so that the union $\bigcup_{j=1}^g (\alpha_j\cup \beta_j)\cup \bigcup_{\ell=1}^{m-1}\gamma_{\ell}$
is a standard bouquet of circles for $X$.

Let  $X$ be a connected finite open smooth surface with base point $q_0$, and $B$ a standard bouquet of circles in $X$. Choose a neighbourhood $N(B)$ of $B$ in $X$ with the following properties. $N(B)$ contains a smoothly bounded open disc $D$ in $X$ around $q_0$,
such that the intersection of each circle in the bouquet with $X\setminus {D}$ is a closed arc with endpoints on the boundary $\partial D$. Denote the arcs by ${\sf s}_j$.
Moreover, $N(B)$ is equal to the union  $D\cup\bigcup_{j=1}^{2g+m-1}V_j$ of $D$ with disjoint sets $V_j$, where each $V_j$ has the following properties.
There exists a
diffeomorphism $\varphi_j$ from a half-open rectangle $\{x+iy \in\mathbb{C}:|x|<{\sf a},|y|\leq{\sf b}\}$ in the plane with sides parallel to the axes onto $V_j$, that maps the horizontal sides into $\partial D$, and the rest of the rectangle into $X\setminus D$. Moreover, its restriction to the vertical symmetry axis of the rectangle (with the direction of the imaginary axis) parameterizes the arc ${\sf s}_j$  (with the orientation of the respective circle of the bouquet).
We call the $V_j$ bands \index{band} attached to $D$ representing ${\sf s}_j$, and
we call  a neighbourhood of the described form a standard neighbourhood of the standard bouquet $B$.
\index{bouquet ! standard neighbourhood of a standard bouquet} \index{$D$} \index{$V_j$} \index{${\sf s}_j$}
\begin{lemm}\label{lemEl2b}
Suppose a standard bouquet $B$ of circles for a connected open Riemann surface $X$ consists of $g$ pairs of curves $\alpha_j$,$\beta_j$, and $m-1$ curves $\gamma_k$ as above, that represent a system of generators of $\pi_1(X,q_0)$, and let $N(B)$ be a standard neighbourhood of $S$. Then  there exists a continuous family of mappings $\varphi_t:X\to X,\, t\in [0,1],$ that map $X$ homeomorphically onto a domain $X_t$ in $X$ such that $\varphi_0$ is the identity and $X_1=N(B)$. In particular,
$X$ has genus $g$ and $m$ holes.
\end{lemm}
\noindent {\bf Proof.}
For each $j$ the pair of loops $\alpha_j$, $\beta_j$ in $N(B)$ has non-vanishing intersection number. Each of these pairs corresponds to a handle of $N(S)$.
For any other pair of loops the intersection number \index{intersection number} equals zero. Each loop $\gamma_k,\, k=1,\ldots ,m-1,$ divides $N(B)$. One of the connected components of $N(B)\setminus \gamma_k$ is an annulus $A_k$ whose second boundary component $\partial_k$ is a boundary component of $N(B)$.

The loop $\gamma_k$ also divides $X$. Otherwise there would be a simple closed curve $\gamma$ in $X$ with base point $q_0$ whose only intersection point with $\gamma_k$ is $q_0$, and that has points on $A_k$ and in $N(B)\setminus A_k$. Then the intersection number of the curves $\gamma$ and $\gamma_k$ would be different from zero. But all elements of the fundamental group of $N(B)$, and hence all elements of the fundamental group of $X$, have vanishing intersection number with $\gamma_k$. Let $\Omega_k$ be the component of $X\setminus \gamma_k$ that  contains
$\partial_k$. Put $X_k=\Omega_k\setminus \overline{A_k}$.

The last boundary component $\partial_m$ of $N(B)$ also divides $X$. Otherwise there would be an arc in $X$ that joins two points that are close to a point in  $\partial_m$ and locally on different sides of  $\partial_m$. Then there would exist a simple closed loop in $X$ with base point $q_0$ that has non-zero intersection number with $\partial_m$. This is impossible, since all elements of the fundamental group $\pi_1(X,q_0)$ can be represented by loops that are contained in $N(B)$ and hence have zero intersection number with   $\partial_m$. Let $X_m$ be the connected component of $X\setminus  \partial_m$, that contains the part of a neighbourhood of $\partial_m$ which is contained in $X\setminus N(B)$.

Each of the $X_k$ is an open Riemann surface of some genus with a number of holes. Since $\alpha_j$, $\beta_j$, and $\gamma_k$ represent generators of the fundamental group of both, $X$ and $N(B)$, each $X_k$ has genus zero. No $X_k$ is a disc. If some $X_k,\, k\leq m-1,$ was a disc, $\gamma_k$ would be contractible in $X$, which is a contradiction. $X_m$ is not a disc, because $\partial_m$ is free homotopic to $\gamma_1\ldots \gamma_{m-1}\alpha_1\beta_1\alpha_1^{-1}\beta_1^{-1}\ldots \alpha_g\beta_g\alpha_g^{-1}\beta_g^{-1}$ which is not the identity in the fundamental group of $X$. Since the fundamental groups of $X$ and $N(B)$ are isomorphic, each $X_k$ is an annulus bounded be a boundary component of $N(B)$ and $\partial_k$. We proved that
$X$ has genus $g$ and $m$ holes. Moreover, there exists a continuous family of mappings $\varphi_t:X\to X,\, t\in [0,1],$ that map $X$ diffeomorphically onto a domain $X_t$ in $X$ such that $\varphi_0$ is the identity and $X_1=N(B)$. \hfill $\Box$
\index{$\varphi_t$}

\smallskip

\noindent{\bf Coverings.}
For the following facts see \cite{Fo}.
By a covering $P:\mathcal{Y} \to \mathcal{X}$ \index{covering} we mean a continuous map $P$ from a connected topological space $\mathcal{Y}$ to  a connected topological space $\mathcal{X}$ such that for each point $x \in \mathcal{X}$ there is a neighbourhood $V(x)$ of $x$ such that the mapping $P$ maps each connected component of the preimage of $V(x)$ homeomorphically onto $V(x)$. (Note that in function theory sometimes these objects are called unlimited unramified coverings to reserve the notion ''covering'' for more general objects.)
A covering $P:\mathcal{Y} \to \mathcal{X}$ is called a universal covering \index{covering ! universal} if
for each covering $P_{\mathcal Z}:\mathcal{Z} \to \mathcal{X}$ of $\mathcal{X}$ by a connected topological space $\mathcal Z$ and for
each pair of points $y\in \mathcal{Y}$ and $z\in \mathcal{Z}$ with $P(y)=P_{\mathcal Z}(z)$
there exists a unique fiber preserving continuous mapping $f:\mathcal{Y}\to \mathcal{Z}$ that takes $y$ to $z$.
A connected locally simply connected topological space has up to isomorphism at most one universal covering.
If  $\mathcal{X}$ and  $\mathcal{Y}$ are connected manifolds, $\mathcal{Y}$ is simply connected, and
$P:\mathcal{Y} \to \mathcal{X}$ is a covering,
then $P$ is the universal covering.

For each connected manifold $\mathcal{X}$ the
universal covering exists and is denoted by  ${\sf P}: \tilde{\mathcal{X}}\to \mathcal{X}$.
It is constructed as follows. Fix a base point $x_0$ in $\mathcal{X}$. For each $x\in \mathcal{X}$ we denote by $\pi_1(\mathcal{X};x_0,x)$ the set of homotopy classes of arcs in $X$ with initial point $x_0$ and terminal point $x$.
Consider the set that consists of the union of all elements $e_{x_0,x} \in \pi_1(\mathcal{X};x_0,x)$ for all $x\in \mathcal{X}$. Consider the mapping that assigns to each element $e_{x_0,x} \in \pi_1(\mathcal{X};x_0,x)$ the point $x\in \mathcal{X}$. We define a system of fundamental neighbourhoods of elements of the defined set
as follows.
For each element $e_{x_0,x} \in \pi_1(\mathcal{X};x_0,x)$ we take any simply connected neighbourhood $U\subset \mathcal{X}$ and consider the set of all products $e_{x_0,x}e_{x,y}\in \pi_1(\mathcal{X};x_0,y)$ for which $e_{x,y}$ can be represented by curves that are contained in contained in $U$ and join $x$ with $y$.
We obtain a simply connected manifold $\tilde{\mathcal{X}}$ and a covering ${\sf P}:\tilde{\mathcal{X}}\to \mathcal{X}$. (For a proof see \cite{Fo}.)

Consider the universal covering ${\sf P}:\tilde{ X}\to X$ of a manifold $X$.
A covering transformation \index{covering transformation} is a homeomorphism $\varphi: \mathcal{Y}\,\toitself$ such that $P\circ \varphi =P$. The covering transformations form a group under composition, denoted by ${\rm Deck}(\tilde{X},X)$. \index{${\rm Deck}(\tilde{X},X)$}
By the definition of the universal covering
for each pair of points $z_1,z_2$ of $\tilde X$ with ${\sf P}(z_1)={\sf P}(z_2)$ there exists a unique element $\varphi\in {\rm Deck}(\tilde{X},X)$ with $z_2=\varphi(z_1)$.
\index{${\sf P}:\tilde{ X}\to X$}

Equip $X$ with a Riemannian metric $d$, and denote by the same letter $d$ the lifted metric on $\tilde X$. In other words, the length of a curve in the metric $d$ on $\tilde X$ is the length of its projection to $X$. 
Take any point $z_0\in\tilde X$. The Dirichlet region \index{region ! Dirichlet} containing $z_0$ is defined as
\begin{equation}
\mathrm{D}_{z_0}\stackrel{def}= \{z\in \tilde X: d
(z,z_0)<d(z,g(z_0)) \;{\rm for \; all } 
\;g \in {\rm Deck}(\tilde{X},X)\,\}\,.
\end{equation}
The Dirichlet region is an open set and $g_1(\mathrm{D}_{z_0})\cap g_2(\mathrm{D}_{z_0})=\emptyset$ if $g_1$ and $g_2$ are different covering transformations.

The union of the closures $\overline{g(\mathrm{D}_{z_0})}$ over all covering transformations $g$ is equal to the whole $\tilde X$. Hence, the union of  $\mathrm{D}_{z_0}$ with a suitable subset of its boundary is a fundamental region. \index{region ! fundamental}
A subset $U$ of $\tilde X$ is called a fundamental region for the covering ${\sf P}:\tilde X \to X$
if each point in $X$ has exactly one preimage under the mapping ${\sf P}:U\to X$.
Further, for each point $z$ of the boundary of $\mathrm{D}_{z_0}$ (taken in the topological space $\tilde X$) there is a covering transformation $g$ such that $d(z,z_0)=d(z,g(z_0))$.

We saw that for each point $z_0\in \tilde X$ there exists a neighbourhood $U$ such that $U\cap g(U)=\emptyset$ for all covering transformations except the identity.
In particular, no element of this group except the identity has a fixed point, and the group
acts discontinuously \index{group action! discontinuous} on $\tilde X$, i.e. for two compact sets $A$ and $B$ in $\tilde X$ the set $A\cap g(B)$ is non-empty for at most finitely many elements $g$ of the group.

Let $X$ be a connected finite open Riemann surface with base point $q_0$ and let ${\sf P}: \tilde X \to X$
be the universal covering map.
We fix a base point $q_0\in X$ and a base point $\tilde{q}_0\in {\sf P}^{-1}(q_0)\subset \tilde{X}$. The group of covering transformations of  $\tilde X$ \index{group ! group of covering transformations} can be identified with the fundamental group $\pi_1(X,q_0)$ of $X$ by the following correspondence (See e.g. \cite{Fo}).
Take a covering transformation $\sigma \in {\rm Deck}(\tilde{X},X)$.
Let $\tilde{\gamma}_0$ be an arc in $\tilde X$ with initial point $\tilde{q}_0$
and terminal point $\sigma(\tilde{q}_0)$.
Denote by ${\rm Is}^{\tilde{q}_0}(\sigma)$  \index{${\rm Is}^{\tilde{q}_0}(\sigma)$}  the element of  $\pi_1(X,q_0)$ represented by the loop ${\sf P}(\tilde{\gamma}_0)$.

By the convention to write products of curves from left to right and compositions of mappings from right to left for the mapping ${\rm Deck}(\tilde{X},X)\ni\sigma\to {\rm Is}^{\tilde{q}_0}(\sigma)\in \pi_1(X,q_0)$ the following relation holds
$$
{\rm Is}^{\tilde{q}_0}(\sigma_1 \sigma_2)={\rm Is}^{\tilde{q}_0}(\sigma_2){\rm Is}^{\tilde{q}_0}(\sigma_1)\,.
$$
Hence, the mapping $\sigma\to {\rm Is}^{\tilde{q}_0}(\sigma^{-1})$
is a group homomorphism. This homomorphism
is injective and surjective, hence it is a group isomorphism. The inverse $({\rm Is}^{\tilde{q}_0})^{-1}$ of the mapping ${\rm Is}^{\tilde{q}_0}$ is obtained as follows. Represent an element $e_0\in \pi_1(X,q_0)$ by a loop $\gamma_0$. Consider the lift $\tilde{\gamma}_0$ of $\gamma_0$ to $\tilde X$ that has initial point $\tilde{q}_0$. Then
$({\rm Is}^{\tilde{q}_0})^{-1}(e_0)$ is the covering transformation that maps $\tilde{q}_0$ to the terminal point of $\tilde{\gamma}_0$.

For another point $\tilde{q}$ of $\tilde X$ and the point $q\stackrel{def}={\sf P}(\tilde{q})\in X$ the isomorphism ${\rm Is}^{\tilde{q}}:{\rm Deck}(\tilde{X},X)\to\pi_1(X,q)$ assigns to each $\sigma \in{\rm Deck}(\tilde{X},X)$ the element of $\pi_1(X,q)$ that is represented by ${\sf P}(\tilde{\gamma})$ for a curve $\tilde{\gamma}$ in $\tilde X$ that joins $\tilde q$ with $\sigma(\tilde{q})$. \index{$\sigma(\tilde{q})$}
${\rm Is}^{\tilde{q}}$ is related to ${\rm Is}^{\tilde{q}_0}$ as follows.
Let $\tilde{\alpha}$ be an arc in $\tilde X$ with initial point $\tilde{q}_0$ and terminal point $\tilde{q}$.
Put
$\alpha={\sf P}(\tilde{\alpha})$.
Then for the isomorphism ${\rm Is}_{\alpha}:\pi_1(X,q_0)\to \pi_1(X,q)$ the equation
\begin{equation}\label{eq1''}
{\rm Is}^{\tilde{q}}(\sigma)=
{\rm Is}_{\alpha}\circ{\rm Is}^{\tilde{q}_0}(\sigma),\;\; \sigma \in {\rm Deck}(\tilde{X},X),
\end{equation}
holds, i.e. the following diagram Figure \ref{fig2} is commutative.
\begin{figure}[H]
\begin{center}
$$
\xymatrix{
&{\pi_1(X,q_0)} \ar[dd]^{{\rm Is}_{\alpha}    } \\
{{\rm Deck}(\tilde{X},X)  } \ar[ru]^{{\rm Is}^{\tilde{q}_0}} \ar[rd]_{{\rm Is}^{\tilde{q}}  } \\
& \pi_1(X,q)
}
$$
\end{center}
\caption{A commutative diagram related to the change of the base point}
\label{fig2}
\end{figure}
Indeed, let
$\tilde{\alpha}^{-1}$ denote the curve that is obtained from
$\tilde{\alpha}$ by inverting the direction on $\tilde{\alpha}$, i.e. moving from $\tilde{q}$ to $\tilde{q}_0$.
For a curve $\tilde{\gamma}_0$ in $\tilde X$ that joins   $\tilde{q}_0$ with $\sigma(\tilde{q}_0)$, the curve
$\tilde{\alpha}^{-1} \; \tilde{\gamma}_0\;\sigma(\tilde{\alpha})$ in $\tilde X$ has initial point $\tilde{q}$ and terminal point $\sigma(\tilde{q})$.
Therefore ${\sf P}(\tilde{\alpha}^{-1} \; \tilde{\gamma}_0\;\sigma(\tilde{\alpha}) )$ represents ${\rm Is}^{\tilde{q}}(\sigma)$.
On the other hand
\begin{equation}\label{eq1'}
{\sf P}(\tilde{\alpha}^{-1} \; \tilde{\gamma}_0\;\sigma(\tilde{\alpha}))=
{\sf P}(\tilde{\alpha}^{-1})\;{\sf P}(\tilde{\gamma}_0)\; {\sf P}(\sigma(\tilde{\alpha})  )=
\alpha^{-1} \gamma_0 \,\alpha
\end{equation}
represents ${\rm Is}_{\alpha}(e_0)$ with $e_0={\rm Is}^{\tilde q_0}(\sigma)$.
In particular, if $\tilde{q}'_0\in {\sf P}^{-1}(q_0)$ is another preimage
of the base point $q_0$ under the projection $\sf P$,
then the associated isomorphisms to the fundamental group $\pi_1(X,q_0)$ are conjugate, i.e. ${\rm Is}^{\tilde{q}'_0}(e_0)=(e'_0)^{-1} {\rm Is}^{\tilde{q}_0}(e_0) e'_0$ for each $e_0\in\pi_1(X,q_0)$. The element $e'_0$ is represented by the projection of an arc in $\tilde X$ with initial point $\tilde{q}_0$ and terminal point $\tilde{q}'_0$.

Keeping fixed $\tilde{q}_0$ and $q_0$ we will say that a point $\tilde{q}\in \tilde X$ and a curve $\alpha$ in $X$ are compatible if the diagram Figure \ref{fig2} is commutative, equivalently, if equation \eqref{eq1''} holds.
We may also start with choosing a curve $\alpha$ in $X$ with initial point $q_0$ and terminal point $q$.
Then there is a point $\tilde{q}=\tilde{q}(\alpha)$, such that $\tilde{q}$ and $\alpha$
are compatible.
Indeed,
let $\tilde{\alpha}$ be the lift of $\alpha$, that has initial point $\tilde{q}_0$.
Denote the terminal point of $\tilde{\alpha}$ by $\tilde{q}(\alpha)$, and repeat the previous arguments.

For two Riemann surfaces $\mathcal{X}$ and $\hat{\mathcal{X}}$ and a non-constant holomorphic mapping $p:\hat{\mathcal{X}}\to \mathcal{X}$ a point $\hat{x}\in \hat{\mathcal{X}}$ is called a branch point or ramification point if there is no neighbourhood of  $\hat{x}$ such the restriction of $p$ to it is injective. For each branch point there exists a neighbourhood $V$ such that $p\mid V$ is equal to the mapping $z\to z^k$  for an integer number $k$ in suitable coordinates on  $\mathcal{X}$ and $\hat{\mathcal{X}}$ near $x$ and $\hat x$.

A non-constant holomorphic mapping $p:\hat{\mathcal{X}}\to \mathcal{X}$ is called a branched covering,
if the image $E$ of the set  $\hat{E}$ of branch points is discrete and
the restriction $p\mid (\mathcal{X}\setminus \hat{E})$ is a covering. The set $E$ is called the branch locus.

The branched covering is called simple if for each point $x\in E$ the fiber $p^{-1}(x)$ contains exactly one
branch point and in a neighbourhood of this branch point the mapping is a double branched covering (i.e. in suitable coordinates it equals $z\to z^2$).
\index{covering ! branched} \index{covering ! simple branched} \index{branch locus}

\smallskip

\noindent {\bf Hyperbolic Riemann surfaces.}
For the following collection of well-known results we refer e.g. to \cite{Lehn}, and \cite{Let}.
By the Uniformisation Theorem the universal covering space $\tilde X$ of each connected Riemann surface $X$ (equipped with the complex structure that makes the projection ${\sf P}:\tilde{X}\to X$ holomorphic) is conformally equivalent to one of the following: the Riemann sphere $\mathbb{P}^1$, or the complex plane $\mathbb{C}$, or the upper half-plane $\mathbb{C}_+\stackrel{def}=\{z\in \mathbb{C}: {\rm Im}z>0\}$. The upper half-plane is equipped with the hyperbolic metric which is defined by the infinitesimal length element $\frac{|dz|}{2y}$. If the upper half-plane is replaced by the unit disc $\mathbb{D}$, \index{hyperbolic metric}
the hyperbolic metric (that is induced by a conformal mapping $\mathbb{C}_+\to \mathbb{D}$) is defined by $\frac{|dz|}{1-|z^2|}$. \index{$\mathbb{C}^*$} \index{$\mathbb{C}_+$} \index{${\rm Im}(z)$}
A Riemann surface whose universal covering is (conformally equivalent to) the upper half-plane will be called hyperbolic.

Let $X$ be a finite Riemann surface with $\tilde X$ being the upper half-plane. Each covering transformation is a conformal mapping of the upper half-plane onto itself. Such a mapping extends to M\"obius transformation on the Riemann sphere. Hence, the group of covering transformations is a subgroup of the projective special linear group $PSL_2(\mathbb{R})=SL_2(\mathbb{R})\diagup \langle-{\rm Id}\rangle $, where $SL_2(\mathbb{R})\diagup \langle -{\rm Id}\rangle $ is the group of mappings $z\to \frac{az+b}{bz+c}$ with complex numbers $a,b,c,d$, that map the real axis to itself, and $\langle-{\rm Id}\rangle$ is the subgroup generated by $-\rm Id$ for the identity $\rm Id$. It acts properly discontinuously.
Subgroups of $PSL_2(\mathbb{R})$ that act properly discontinuously are called Fuchsian groups. For each fixed point free Fuchsian group $\Gamma$ the quotient $\mathbb{C}_+\diagup \Gamma$ is a Riemann surface.

The limit set of a Fuchsian group $\Gamma$ is the set of accumulation points of $\{g(z_0): g\in \Gamma\}$, where $z_0$ is any point of the upper half-plane. The set is independent of the choice of the point $z_0$. The limit set of a fixed point free Fuchsian group is either the real axis $\mathbb{R}$ or a nowhere dense subset of the real axis without isolated points.
A Fuchsian group is said to be of first kind if the limit set is the real axis, and is said to be of second kind if the limit set is a nowhere dense subset of the real axis without isolated points.
\index{group ! Fuchsian group}
The following theorem motivates our definition of Riemann surfaces of first and second kind.
\begin{thm}\label{prop2.2*}
A connected Riemann surface is of first kind iff it is the quotient $\mathbb{C}_+\diagup \Gamma$ of the upper half-plane by a
Fuchsian group $\Gamma$ of first kind.
\end{thm}
For a proof (in a more general situation) see e.g. \cite{Kra}, II Theorem 3.2 and \cite{A2}.

\noindent {\bf The universal covering of the twice punctured complex plane.}
We recall the well-known explicit description of the universal covering of $\mathbb{C}\setminus \{0,1\}$. We consider the domain $D_0=\{z\in \mathbb{C}_+: 0<{\rm Re}z<1, |z-\frac{1}{2}|>\frac{1}{2}\}$ in the upper half-plane, which is bounded by a half-circle and two half-lines. These curves are geodesic curves in the hyperbolic metric of the upper half-plane. We call $D_0$ the geodesic triangle  with \index{geodesic triangle}
vertices $0,\,1,$ and $\infty$. Denote by  $\sf P$ the conformal mapping from $D_0$ onto the upper half-plane, whose continuous extension to the boundary
takes $0$ to $0$, $1$ to $1$ and $\infty$ to $\infty$. The extension of $\sf P$ to the boundary takes $\ell\stackrel{def}=(0, i\infty)$ to $(-\infty,0)$, $\mathcal{C}\stackrel{def}= \{|z-\frac{1}{2}|=\frac{1}{2}\}\cap \mathbb{C}_+$ to $(0,1)$ and $r\stackrel{def}=1+(0,+i\infty)$ to $(1,\infty)$.
\index{${\rm Re}(z)$}

We reflect $D_0$ in the circle  $|z-\frac{1}{2}|=\frac{1}{2}$. By symmetry reasons the reflected domain is the geodesic triangle $D_1$ bounded by $\mathcal{C}$ (the reflection of itself), and the half-circles in $\mathbb{C}_+$ with diameter $(0,\frac{1}{2})$ (the reflection of $\ell$), and diameter $(\frac{1}{2},1)$ (the reflection of $r$). In the same way we reflect $D_0$ in the half-lines in the boundary.

\begin{figure}[H]
\begin{center}
\includegraphics[width=60mm]{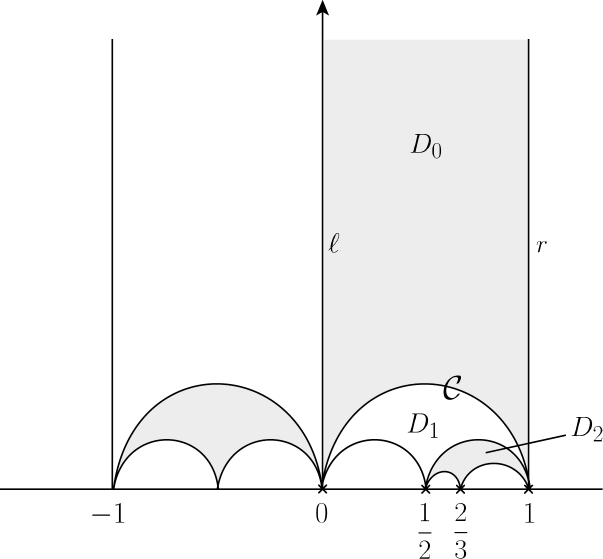}
\end{center}
\caption{The universal covering of the twice punctured complex plane}
\end{figure}

We reflect now $D_1$ in the circle  $|z-\frac{3}{4}|=\frac{1}{4}$ with diameter $(\frac{1}{2},1)$.
We obtain the geodesic triangle $D_2$ with vertices $\frac{1}{2}$, $\frac{2}{3}$,
and $1$. This can be seen by the following calculation. The mapping $z\to\zeta(z)= 4(z-\frac{3}{4})$ maps the circle $|z-\frac{3}{4}|=\frac{1}{4}$ to the unit circle
$|\zeta|=1$. The reflection map in $|\zeta|=1$ is $\zeta\to \frac{1}{\overline{\zeta}}$, hence, the reflection in  the circle $|z-\frac{3}{4}|=\frac{1}{4}$  is defined by the mapping
$$
z=\frac{1}{4}(\zeta+3)\to \frac{1}{4}(\frac{1}{\overline{\zeta}}+3)=\frac{1}{4}(\frac{1}{4\bar{z}-3}   +3)=
\frac{3\bar{z}-2}{4\bar{z}-3}\,.
$$
The mapping takes $1$ to $1$, $0$ to $\frac{2}{3}$, and $\frac{1}{2}$ to $\frac{1}{2}$.
We get further geodesic triangles by continuing reflecting in each geodesic arc of each already obtained geodesic triangle.
We get a collection of open geodesic triangles whose closure in $\mathbb{C}_+$ fill the whole $\mathbb{C}_+$.

Extend the mapping $\sf P$ continuously to the closure of $D_0$, and
by Schwarz' Reflection Principle through the geodesic arcs in the boundary of $D_0$.  Extending $\sf P$ repeatedly through the geodesic boundary arcs of the images under repeated reflection, we obtain a holomorphic mapping form $\mathbb{C}_+$ to $\mathbb{C}\setminus \{0,1\}$, denoted also by ${\sf P}$. The extended $\sf P$ takes each geodesic triangle, obtained from $D_0$ by an even number of successive reflections, conformally onto the upper half-plane, and maps each geodesic triangle, obtained from $D_0$ by an odd number of successive reflections, conformally onto the lower half-plane. It is easy to see that ${\sf P}: \mathbb{C}_+\to \mathbb{C}\setminus \{-1,1\}$ is a covering, in fact it is the universal covering of the twice punctured complex plane.
\index{Schwarz ! reflection principle}

Each reflection in a geodesic side of one of the geodesic triangles obtained from $D_0$ by reflection is an anti-holomorphic self-homeomorphism of $\mathbb{C}_+$. The composition $\varphi$ of an even number of successive reflections is a conformal mapping of $\mathbb{C}_+$ onto itself, for which ${\sf P}\circ \varphi= {\sf P}$, in other words, it is a covering transformation. The covering transformation obtained by reflection in the imaginary axis followed by reflection in one of the half-circles $\{z\in \mathbb{C}_+: |z\pm\frac{1}{2}|=\frac{1}{2}\}$ equals $z\to \frac{z}{1\mp 2 z}$.
Indeed, the reflection in $\{z\in \mathbb{C}_+: |z\pm\frac{1}{2}|=\frac{1}{2}\}$
is given by the mapping $z\to \frac{\mp\bar{z}}{2\bar{z}\pm 1}$, the reflection in the imaginary axis is given by $x+iy\to -x+iy$.

Reflection in the imaginary axis followed by reflection in $\{z\in \mathbb{C}_+: {\rm Re} z=1\}$ equals $z\to z+2$. These covering transformation generate the group of covering transformations. In particular,
each covering transformation is an element of $PSL_2(\mathbb{Z})$. Here
$SL_2(\mathbb{Z})$ is the group of $2\times 2$ matrices with integer entries and determinant $1$, and
$PSL_2(\mathbb{Z})$ is the quotient $SL_2(\mathbb{Z})\diagup \langle -{\rm Id}\rangle $, for the subgroup $\langle -{\rm Id}\rangle $ where $\rm Id$ is the unit matrix. \index{$SL_2(\mathbb{Z})$} \index{$PSL_2(\mathbb{Z})$}

\medskip

\noindent {\bf Properties of elements of $PSL_2(\mathbb{Z})$.}
Let $T(z)=\frac{az+b}{cz+d}$ represent an element of $PSL_2(\mathbb{Z})$. The fixed points of this mapping satisfy the equation $\frac{az+b}{cz+d}=z$, hence, if $c\neq 0$, $z^2+\frac{d-a}{c}z -\frac{b}{c}=0$. This implies that the fixed points are $z_{\pm}= -\frac{d-a}{2c} \pm \frac{\sqrt{(a+d)^2-4}}{2c}$. We used the fact that the determinant of the matrix
$A=\begin{pmatrix}
a & b  \\
c & d
\end{pmatrix}$ equals $ad-bc=1$.
The eigenvalues $t_{\pm}$ of the matrix $A$ satisfy the equation $t^2-(a+d)t+1=0$, hence, $t_{\pm}=\frac{a+d}{2} \pm \frac{\sqrt{(a+d)^2-4}}{2}$. We used again that the determinant of the matrix equals $1$.

If $|a+d|>2$ there are two different real fixed points. The matrix $A$ has
two different real eigenvalues $t_{\pm}$, and can be conjugated by a real matrix to the matrix $\begin{pmatrix}
t_+ & 0  \\
0   & t_-
\end{pmatrix}$. Hence, $T$ can be conjugated by a holomorphic self-homeomorphism of $\mathbb{C}_+$ to the mapping $T'(z)=\frac{t_+z}{t_-}= t_+^2 z$.

If $|a+d|=2$, $T$ has a double fixed point $z_0$. This point is real. Conjugating this fixed point
to $\infty$ (for instance by the conformal self-mapping $z\to \frac{1}{z_0-z}$ of $\mathbb{C}_+$) we arrive at a mapping of the form $z\to z+{\sf b}$ for a real number ${\sf b}$. The eigenvalues are equal to $t_{\pm}=1$.
If $c=0$ then $a=d=\pm 1$, hence $|a+d|=2$.

If $|a+d|<2$ the eigenvalues are equal to $t_{\pm}=\pm i$, if $|a+d|=0$,  and to $t_{\pm}=\frac{1}{2}\pm \frac{\sqrt{3}}{2}i$ or  $t_{\pm}=-\frac{1}{2}\pm \frac{\sqrt{3}}{2}i$, if $|a+d|=1$. In the first case the square of the mapping $T^2$ is the identity. In the second case $T^3$ is the identity.

\section{Braids}
\label{sec:2.1}

In this section we outline basic facts on braids.
For details the reader may consult \cite{Bi} or \cite{KaTu}.

Most intuitively, braids are described in terms of geometric braids.
We will use here the complex plane ${\mathbb C}$
\index{$\mathbb{C}$} though the complex structure will play a role
only later.
Let $E_n$
be a subset of ${\mathbb C}$ containing exactly $n$ points.
A geometric braid \index{braid ! geometric} in
$[0,1] \times {\mathbb C}$ with base point $E_n$ is a collection of
$n$ mutually disjoint arcs in the  cylinder $[0,1] \times {\mathbb
C}$ which joins the set $\{ 0 \} \times E_n$ in the bottom with the
set $\{ 1 \} \times E_n$ in the top of the cylinder and
intersects each fiber $\{ t \} \times {\mathbb C}$ along an
unordered $n$-tuple $E_n(t)$ of points. The arcs are called the
strands of the geometric braid. \index{braid ! strand of a braid}

\index{tuples ! ordered  $(z_1 , \ldots , z_n)$}
\index{tuples ! unordered $\{z_1, \ldots, z_n
\}$}

There is a more conceptional point of view.
We consider points $(z_1 , \ldots , z_n) \in {\mathbb C}^n$
as ordered tuples of points in ${\mathbb C}$. Usually the
points $z_j \in {\mathbb C}$ will be required to be pairwise
distinct, in other words, we require that $(z_1 , \ldots , z_n)$
belongs to the configuration space $C_n ({\mathbb C}) = \{(z_1
, \ldots , z_n) \in {\mathbb C}^n : z_i \ne z_j$ for $i \ne j\}$ of $n$ particles moving in the plane without collision.
\index{configuration space} Denote by ${\mathcal S}_n$
\index{$\mathcal{S}_n$} the symmetric group \index{group ! symmetric group}.
Each permutation in ${\mathcal S}_n$ acts on $C_n ({\mathbb C})$ by
permuting the coordinates. The quotient $C_n ({\mathbb C}) \diagup
{\mathcal S}_n$ \index{$C_n ({\mathbb C}) \diagup {\mathcal S}_n$}
is called the symmetrized configuration space. \index{symmetrized
configuration space} We denote its elements by $\{z_1, \ldots, z_n
\}$ and consider them as unordered $n$-tuples of points in ${\mathbb
C}$ or as subset $E_n$ of ${\mathbb C}$ consisting of exactly $n$
points. For a subset $A$ of ${\mathbb C}$ we will put $C_n (A) = C_n
({\mathbb C}) \cap A^n$.


The configuration space inherits the topology and complex structure
from ${\mathbb C}^n$. The symmetrized configuration space is given
the quotient topology and quotient complex structure. Note that
${\mathcal S}_n$ acts freely and properly discontinuously on $C_n
({\mathbb C})$. The canonical projection ${\mathcal P}_{\rm sym} :
C_n ({\mathbb C}) \to C_n ({\mathbb C}) \diagup {\mathcal S}_n$
\index{$\mathcal{P}_{\rm sym}$}  \index{$C_n ({\mathbb C}) \diagup
{\mathcal S}_n$} is holomorphic. \index{$\mathcal{P}_{\rm sym}$}


Identifying subsets of $\mathbb{C}$ consisting of $n$ points
with elements of  $C_n(\mathbb{C})\diagup\mathcal{S}_n$, we may regard a geometric braid
with base point $E_n$ as a set of the form
\begin{equation}\label{eq2.0}
\{(t,f(t)),\, t \in[0,1]\} \subset [0,1]\times \mathbb{C}
\end{equation}
for a continuous mapping $[0,1] \ni t
\overset{f}{\longrightarrow} C_n ({\mathbb C}) \diagup {\mathcal
S}_n$ with $f(0) = f(1) = E_n$.
We will often identify the geometric braid \eqref{eq2.0} with the mapping $f$ defining it.
Two geometric braids with base point
$E_n$ are called isotopic if there is a continuous family of
geometric braids with base point $E_n$ joining them. A braid
\index{braid} on $n$ strands with base point $E_n$ is an isotopy
class of geometric braids with base point $E_n$. Isotopy classes of
geometric braids \index{isotopy class ! of geometric braids} with
base point $E_n$ form a group. The operation is obtained by putting
one geometric braid on the top of another.


Take a geometric braid with base point $E_n$. Assigning to each
point in $\{ 0 \} \times E_n$ the point in $\{ 1 \} \times E_n$
which belongs to the same strand we obtain a permutation of the points of $E_n$.
It depends only on the isotopy class which we denote by $b$. This
permutation is denoted by $\tau_n (b)$. \index{$\tau_n(b)$} The
braid $b$ is called pure if $\tau_n (b)$ is the identity.
The set of pure $n$-braids forms a group which we denote by $\mathcal{PB}_n$.
\index{braid ! pure} \index{$\mathcal{PB}_n$}


Algebraically, $n$-braids are represented as elements of the Artin braid \index{group ! Artin braid group}
group ${\mathcal B}_n$.
This is the group
with generators denoted by $\sigma_1 , \ldots , \sigma_{n-1}$ and
relations $\sigma_i \, \sigma_j = \sigma_j \, \sigma_i$ for $\vert
i-j \vert \geq 2$, and $\sigma_i \, \sigma_{i+1} \, \sigma_i =
\sigma_{i+1} \, \sigma_i \, \sigma_{i+1}$ for $i = 1,\ldots , n-2$.
\index{$\sigma_i$}
An isomorphism between ${\mathcal B}_n$ and
isotopy classes of geometric braids with given base point $E_n$ can be
obtained as follows. Choose a projection ${\rm Pr}$ of ${\mathbb C}$ \index{${\rm Pr}$}
onto the real line ${\mathbb R}$ which is injective on $E_n$.
Label the points in $E_n$ so that $E_n=\{z_1,\ldots,z_n\}$ with ${\rm Pr}(z_j)<{\rm Pr}(z_{j+1})$. The strand of a (geometric) braid with initial point $z_j$ is called the $j$-th strand.
The
projection $[0,1] \times {\mathbb C} \ni (t,z) \to (t, {\rm Pr} \,
z) \in [0,1] \times {\mathbb R}$ assigns to each geometric braid a
braid diagram (at intersection points of images of strands it is
indicated which strands is ``over'' and which strand is ``under'').
Isotopy classes of geometric
braids with base point $E_n$ correspond to equivalence classes of diagrams. The equivalence classes of diagrams can
be interpreted as elements of the Artin braid group.
The braid diagram corresponding to the generator $\sigma_j$ of the Artin braid group is the diagram with one intersection point, at this point the $j$-th strand is over the $j+1$-st strand.

The isomorphism from the group of isotopy classes of geometric
braids to the Artin group depends on the base point $E_n$ and the
projection ${\rm Pr}$. A change of the base point
and a change of the projection lead to
conjugation with an element of the Artin group.
Indeed, let $\alpha$ be a curve in $C_n(\mathbb{C})\diagup \mathcal{S}_n$
with initial point $E_n$ and endpoint $E'_n$. A change of base point from $E_n$ to $E'_n$ along $\alpha$ assigns to each curve $\gamma$ in $C_n(\mathbb{C})\diagup \mathcal{S}_n$
with initial point and endpoint equal to $E_n$ (i.e. to each geometric braid with base point $E_n$) the curve $\alpha^{-1} \gamma \alpha$.
The latter curve is obtained by travelling first along the arc $\alpha^{-1}$ in the symmetrized configuration space, that joins the new base point $E_n'$ with $E_n$, then along the curve $\gamma$, and then along $\alpha$. Looking at the associated braid diagram using a fixed projection (provided it is injective on both, $E_n$ and $E'_n$) and considering equivalence classes we obtain the statement concerning the change of base point. A change of the projection can be made by fixing the projection and considering an isotopy of the space $\mathbb{C}$.
This leads to an isotopy of geometric braids, hence we obtain the statement for the change of projections.

The center $\mathcal{Z}_n$ of the Artin braid group $\mathcal{B}_n$ is generated by the element
$\Delta_n^2$, where $\Delta_n\stackrel{def}= (\sigma_1 \sigma_2\ldots \sigma_{n-1})(\sigma_1 \sigma_2\ldots \sigma_{n-2}) \ldots (\sigma_1 \sigma_2) \sigma_1 $ is called the Garside element. The braid $\Delta_n$ is
represented by the following geometric braid. Choose a set $E_n\subset \mathbb{R}$
that is invariant under reflection in the imaginary axis. The representing geometric braid is
obtained from the trivial geometric braid with base point $E_n$  by a half-twist of the cylinder $\mathbb{C}\times [0,1]$, i.e. by fixing the bottom of the cylinder and turning the top by the angle $\pi$. We will say that the braid $\Delta_n$
corresponds to a half-twist.
The braid   $\Delta_n^2$ corresponds to a full twist. For an element $e$ of a group we denote by $\langle e\rangle$ the subgroup that is generated by $e$. We put  $\mathcal{Z}_n=\langle\Delta_n^2\rangle$.
\index{ $\langle e\rangle$}\index{$\Delta_n$} \index{Garside element}

The pure braid group for  $n=3$ is related to a free group, more precisely
the following lemma holds.
\begin{lemm}\label{lemm1a} The group $\mathcal{PB}_3\diagup \langle
\Delta_3^2 \rangle$ is isomorphic to the fundamental group of
$\mathbb{C} \setminus \{-1,1\}$ with base point $0$. The generators $\sigma_j^2\diagup \langle\Delta_3^2 \rangle$ of $\mathcal{PB}_3\diagup \langle
\Delta_3^2 \rangle$ correspond to the generators $a_j$ of $\pi_1(\mathbb{C} \setminus \{-1,1\},0)$.
\end{lemm}
Here $a_1$ is represented by a loop that surrounds $-1$ counterclockwise, and $a_2$ is represented by  a loop that surrounds $1$ counterclockwise.

\medskip

\noindent {\bf Proof.}
For a point $z=(z_1,z_2,z_3)\in C_3(\mathbb{C})$ we denote by $M_z=M_{(z_1,z_2,z_3)}$ the
M\"{o}bius
transformation that maps $z_1$ to $0$, $z_3$ to $1$ and fixes
$\infty$. Then $M_z(z_2)$ omits $0$, $1$ and $\infty$. Notice
that $M_z(z_2)$
 is equal to the cross ratio $(z_2,z_3;z_1,\infty)
=\frac{z_2-z_1}{z_3-z_1}\cdot
\frac{z_3-\infty}{z_2-\infty}=\frac{z_2-z_1}{z_3-z_1}$.
We define $\mathfrak{C}(z)\stackrel{def}=2 M_z(z_2)-1$ with $z=(z_1,z_2,z_3) \in C_3(\mathbb{C})$.  The mapping $\mathfrak{C}$ takes
$C_3(\mathbb{C})$ to $\mathbb{C}\setminus \{-1,1\}$ and
$C_3(\mathbb{R})^0 \stackrel{def}{=}
\{(x_1,x_2,x_3) \in \mathbb{R}^3: x_1<x_2<x_3\}    $ to $(-1,1)$. \index{$\mathfrak{C}(z)$}
Associate to a curve $\tilde{\gamma}(t)\,=\, (\tilde \gamma_1(t),\tilde \gamma_2(t), \tilde \gamma_3(t)),\,
t \in
[0,1],$ in $C_3(\mathbb{C})$ the curve
$\mathfrak{C}(\tilde \gamma)(t)\overset{def}= 2
\,\frac{\tilde \gamma_2(t)-\tilde \gamma_1(t)}{\tilde \gamma_3(t)-\tilde \gamma_1(t)}
\,-1\,,\,
t \in
[0,1],$ in $\mathbb{C}$  
which omits the points $-1$ and $1$. The mapping $\tilde{\gamma}\to \mathfrak{C}(\tilde{\gamma})$ defines a homomorphism  $\mathfrak{C}_*$ from the fundamental group
$\pi_1(C_3(\mathbb{C}) , (-1,0,1))$ of
$C_3(\mathbb{C})$ with base point $(-1,0,1)$ to the fundamental group
$\pi_1 \stackrel{def}{=}\pi_1(\mathbb{C}\setminus
\{-1,1\},0)$
of the twice punctured complex plane with base point $0$.
\index{$\mathfrak{C}(z)$} \index{$M_{(z_1,z_2,z_3)}$} \index{$\mathfrak{C}_*$}

Identify the pure braid group $\mathcal{PB}_3$ with a subgroup of the fundamental group of $C_3(\mathbb{C}) \diagup \mathcal{S}_3$ with base point $\{-1,0,1\}$. For curves $\gamma$ representing
elements of this subgroup we consider the lift $\tilde\gamma$ under $\mathcal{P}_{\rm sym}$ with base point
$(-1,0,1) \in C_3(\mathbb{C})$. This gives an isomorphism from  the pure braid group $\mathcal{PB}_3$
onto the fundamental group of  $C_3(\mathbb{C})$ with base point $(-1,0,1)$.

The braids represented by the two loops
$\tilde{\gamma}\,$ and $\mathring{\gamma}$,
$\mathring{\gamma} (t)\stackrel{def}{=}(-1,\mathfrak{C}(\tilde{\gamma})(t),1) = \big(\,2M_{\tilde{\gamma}(t)}(\tilde{\gamma}_1(t))-1\,,\, 2M_{\tilde{\gamma}(t)}(\tilde{\gamma}_2(t))-1\,,\,2M_{\tilde{\gamma}_3(t)}(\tilde{\gamma}_1(t))-1)\, \big)       ,
\;\,t \in [0,1],\,$ differ by a power $\Delta_3^{2N}$ of the
full
twist $\Delta_3^2$. Indeed, to obtain $\mathring{\gamma}$ we act  for each $t$ by a complex linear mapping on the point $\tilde{\gamma}(t)$.
The number $N$ can be interpreted as
the linking number of the first and the third strands of the
geometric braid $\tilde{\gamma}(t), \, t \in [0,1].\,$ This linking
number is obtained as follows. Discard the second strand. The resulting
braid
equals $\sigma ^{2N}$ where $N$ is the mentioned linking
number. For
the
geometric braid
$\mathring \gamma$ the linking number of
the
first and third strand is zero. It follows that $\mathfrak{C}_*$ is
surjective
and its kernel equals $\langle \Delta_3^2 \rangle$.

\smallskip
A word in the generators of a free group is called reduced, if neighbouring terms are powers of different generators.
We saw that for each element $\mathbold{b}$ of
the pure braid group modulo its center $\mathcal{PB}_3 \diagup \mathcal{Z}_3$
there is a unique element $b\in\mathcal{PB}_3$ that represents $\mathbold{b}$ and can be written as reduced word in $\sigma_1^2$ and $\sigma_2^2$.
Indeed, this representative $b$ of $\mathbold{b}$ is determined by the property that the linking number between the first and the third strand equals zero.
Assigning to each element $\mathbold{b}\in\mathcal{PB}_3 \diagup \mathcal{Z}_3$ the mentioned element we obtain the isomorphism to the free group in two generators $\sigma_1^2$ and $\sigma_2^2$, or, equivalently to
the fundamental group $\pi_1(\mathbb{C} \setminus\{-1,1\},0)$ of the twice punctured complex plane with base point $0$ with generators $a_1$ and $a_2$, respectively. \hfill $\Box$

\smallskip
Geometric braids \index{braid ! geometric} with base point $E_n$
were interpreted as paths in $C_n ({\mathbb C}) \diagup {\mathcal
S}_n$
with initial and terminal point equal
to $E_n$, in other words, as loops in this space with base point
$E_n$. Isotopy classes of geometric braids with base point $E_n$
correspond to homotopy classes of loops in $C_n ({\mathbb C}) \diagup {\mathcal S}_n$
with
base point $E_n$, in other words, braids with base point $E_n$ correspond to elements of the fundamental
group $\pi_1(C_n ({\mathbb C}) \diagup {\mathcal
S}_n, E_n)$
\index{$\pi_1(C_n ({\mathbb C}) \diagup {\mathcal
S}_n ,E_n)$}
of $C_n ({\mathbb C}) \diagup {\mathcal
S}_n$  with base
point $E_n$. Thus the Artin braid group ${\mathcal B}_n$ is isomorphic to $\pi_1
( C_n ({\mathbb C}) \diagup {\mathcal
S}_n  , E_n)$. We saw that a change of the base point leads to an
automorphism of ${\mathcal B}_n$ defined by conjugation with an
element of ${\mathcal B}_n$. Denote by $\widehat{\mathcal B}_n$
\index{$\widehat{\mathcal{B}}_n$} the set of conjugacy classes of
${\mathcal B}_n$. Its elements $\hat b \in \widehat{\mathcal B}_n$ can
be interpreted as free homotopy classes of loops in $C_n ({\mathbb C}) \diagup {\mathcal
S}_n$ . \index{$\mathfrak{P}_n$}
In other words, two geometric braids
$f_0: [0,1] \to C_n ({\mathbb C}) \diagup {\mathcal
S}_n,\,f_1:
[0,1] \to C_n ({\mathbb C}) \diagup {\mathcal
S}_n,\,$
represent the same class $\widehat {\mathcal{B}}_n,\,$ iff there is a
free homotopy joining them, i.e. if there exists a continuous mapping
$h: [0,1] \times [0,1] \to C_n ({\mathbb C}) \diagup {\mathcal
S}_n $ such that
$h(t,0)=f_0(t),\; h(t,1)=f_1(t),\;$ for $t \in [0,1],\;$ and
$h(s,0)=h(s,1)\;$ for $s \in [0,1].\,$ We may consider the
continuous family of braids $f_s:[0,1] \to C_n ({\mathbb C}) \diagup {\mathcal
S}_n$ with
variable base point as a free isotopy of braids. \index{braids !
free isotopy of}

We will also use the following terminology. A loop in $C_n ({\mathbb C}) \diagup {\mathcal S}_n$
(i.e. a continuous map of the circle into
$C_n ({\mathbb C}) \diagup {\mathcal
S}_n)$
is called a closed geometric braid.\index{braid ! closed geometric}
A free homotopy class of loops in $C_n ({\mathbb C}) \diagup {\mathcal
S}_n$ is called a
closed braid. \index{braid ! closed} Closed braids correspond to
elements of $\widehat{\mathcal B}_n$.

Arnol'd interpreted the symmetrized configuration space $C_n
({\mathbb C}) \diagup {\mathcal S}_n$ as the space of monic
polynomials ${\mathfrak P}_n$ \index{$\mathfrak{P}_n$}
of degree $n$ without multiple zeros. (''Monic'' refers to the property that
the coefficient by the highest power of the variable equals $1$.)
Denote by $\overline{\mathfrak
P}_n$ \index{$\overline{\mathfrak{P}}_n$} the set of all monic
polynomials of degree $n$. If we assign to each unordered $n$-tuple
$E_n = \{ z_1 , \ldots , z_n \}$ (this time the $z_j$ are not
necessarily pairwise distinct) the monic polynomial
$\overset{n}{\underset{j=1}{\prod}} (z-z_j) = a_0 + a_1 \, z +
\ldots z^n$ whose set of roots equal $E_n$, we obtain a bijection
onto the set $\overline{\mathfrak P}_n$. The set of monic
polynomials of degree $n$ can also be parameterized by the ordered
$n$-tuple $(a_0 , a_1 , \ldots , a_{n-1}) \in {\mathbb C}^n$ of the
coefficients. Hence, we obtain a bijection between $C_n ({\mathbb
C}) \diagup {\mathcal S}_n$ and the set of points $(a_0 , a_1 ,
\ldots , a_{n-1}) \in {\mathbb C}^n$ which are coefficients of
polynomials without multiple zeros. There is a polynomial
$\textsf{D}_n$ \index{${\sf D}_n$} in the variables $a_0 , \ldots
, a_{n-1}$, called the discriminant, which vanishes exactly if the \index{discriminant}
polynomial with these coefficients has multiple zeros. Hence, the
symmetrized configuration space $C_n ({\mathbb C}) \diagup {\mathcal
S}_n$ can be identified with ${\mathbb C}^n \backslash
V_{\textsf{D}_n}$ where $V_{ \textsf{D}_n} \overset{\rm def}{=} \{
(a_0 , \ldots , a_{n-1}) \in {\mathbb C}^n : \textsf{D}_n (a_0 ,
\ldots , a_{n-1}) = 0\}$. \index{$V_{\textsf{D}_n}$} The
identification is actually a biholomorphic map. Thus, $C_n ({\mathbb
C}) \diagup {\mathcal S}_n$ is biholomorphic to a pseudoconvex
domain in ${\mathbb C}^n$, even stronger, it is biholomorphic to the
complement of an algebraic hypersurface in ${\mathbb C}^n$. The space
${\mathfrak P}_n \cong C_n ({\mathbb C}) \diagup {\mathcal S}_n$
received much attention in connection with problems of algebraic
geometry. For instance, motivated by his interest in the Thirteen's
Hilbert problem, Arnol'd \cite{Ar} computed its topological
invariants. \index{Thirteen's
Hilbert problem}

\section{Mapping class groups.}
\label{sec:2.2}

Mapping class groups are defined as follows.
Let $A$
be a topological space, not necessarily compact, but paracompact.
Let $A_1$ and $A_2$ be disjoint closed subsets of $A$. Denote by
${\rm Hom} (A ; A_1 , A_2)$ \index{${\rm Hom} (A ; A_1 , A_2)$} the
set of self-homeomorphisms of $A$ that fix $A_1$ pointwise and $A_2$
setwise. We will also write ${\rm Hom} (A ; A_1)$  \index{${\rm Hom}
(A ; A_1)$} for ${\rm Hom} (A ; A_1 , \emptyset)$, but in case we do
not require that some points are fixed we write ${\rm Hom} (A ;
\emptyset , A_2)$. We will also write ${\rm Hom} (A )$ \index{${\rm
Hom} (A )$} for ${\rm Hom} (A ; \emptyset , \emptyset)$. Equip the
set ${\rm Hom} (A ; A_1 , A_2)$ with compact open topology.
 \index{mapping class group}

Let $A$ be an oriented manifold. By ${\rm Hom}^+ (A;A_1,A_2)$
\index{${\rm Hom}^+ (A;A_1,A_2)$} we denote the set of orientation
preserving self-homeomorphisms in ${\rm Hom}(A;A_1,A_2)$. It forms a
group with respect to composition. The set of connected components
of ${\rm Hom}^+ (A;A_1,A_2)$ equipped with the inherited group structure is the mapping class group
${\mathfrak M}(A;A_1,A_2)$ corresponding to the triple $(A,A_1,A_2)$.
\index{$\mathfrak{M} (A;A_1,A_2)$}


Let $\bar {\mathbb{D}}$ \index{$\bar {\mathbb{D}}$}  be the closed unit disc in the complex plane ${\mathbb
C}$ with boundary $\partial \mathbb{D}$ and interior $\mathbb{D}$ \index{$\mathbb{D}$}.
At this point we may forget the complex structure of $\mathbb{D}$, but we will need it later.

Let $E_n^0 = \left\{ 0,\frac1n , \ldots , \frac{n-1}n \right\}$  \index{$E_n^0$}
be the ``standard'' subset of $\mathbb{D}$ containing $n$ points. We also
identify $E_n^0$ with the respective unordered $n$-tuple of points.
The set ${\mathfrak M} (\bar {\mathbb{D}};\partial \mathbb{D},E_n^0)$
\index{$\mathfrak{M} (\bar {\mathbb{D}};\partial {\mathbb{D}},E_n^0)$} is commonly known
as the mapping class group of the $n$-punctured disc. Note that
${\mathfrak M} (\bar {\mathbb{D}} ; \partial {\mathbb{D}} ,E_n^0)$ is isomorphic to
${\mathfrak M} (\bar {\mathbb{D}} \backslash E_n^0 ; \partial {\mathbb{D}})$
\index{$\mathfrak{M} (\bar {\mathbb{D}} \backslash E_n^0 ; \partial {\mathbb{D}})$} since
each element of ${\mathfrak M} (\bar {\mathbb{D}} \backslash E_n^0 ; \partial
{\mathbb{D}})$ extends continuously to each point of $E_n^0$ (the
``punctures''). The points of $E_n^0$ are also called
``distinguished points''. The homeomorphisms in ${\rm Hom}^+ (\bar {\mathbb{D}}
; \partial {\mathbb{D}} , E_n^0)$ are called homeomorphisms of $\bar {\mathbb{D}}$ with
distinguished points $E_n^0$. By an abuse of language we will also call them homeomorphisms of the $n$-punctured
disc, having in mind the identification with their restrictions to $\mathbb{D}\setminus E^0_n$. The connected component ${\rm Hom}^0 (\bar {\mathbb{D}} ;
\partial {\mathbb{D}} ,E_n^0)$ \index{${\rm Hom}^0 (\bar {\mathbb{D}} ;
\partial {\mathbb{D}} ,E_n^0)$} of ${\rm Hom}^+ (\bar {\mathbb{D}} ;
\partial {\mathbb{D}} ,E_n^0)$ containing the identity consists of the
self-homeomorphisms of $\bar {\mathbb{D}}$ which can be joined to the identity
by a continuous family of homeomorphisms in ${\rm Hom}^+ (\bar {\mathbb{D}} ;
\partial {\mathbb{D}} ,E_n^0)$. In other words, it consists of homeomorphisms
in ${\rm Hom}^+ (\bar {\mathbb{D}} ; \partial {\mathbb{D}} ,E_n^0)$ which are isotopic to
the identity through homeomorphisms fixing $\partial {\mathbb{D}}$ and $E_n^0$
pointwise. We have
\begin{equation}
\label{eq2.1} {\mathfrak M} (\bar {\mathbb{D}} ; \partial {\mathbb{D}} ,E_n^0) = {\rm
Hom}^+ (\bar {\mathbb{D}} ,
\partial {\mathbb{D}} ,E_n^0) \diagup  {\rm Hom}^0 ({\mathbb{D}} ; \partial {\mathbb{D}} ,E_n^0) \, .
\end{equation}
\index{distinguished points} \index{punctures}
The respective mapping class group ${\mathfrak M} (\bar {\mathbb{D}} ; \partial
{\mathbb{D}} ,E_n)$ \index{$\mathfrak{M} (\bar {\mathbb{D}} ; \partial {\mathbb{D}} ,E_n)$} can be
defined for any unordered $n$-tuple $E_n$ of points in ${\mathbb{D}}$.

\smallskip

For a
homeomorphism $\varphi \in {\rm Hom}^+ (\bar {\mathbb{D}} ;
\partial {\mathbb{D}} , E_n)$ we consider its restriction $\varphi \mid {\mathbb{D}} \in
{\rm Hom}^+ ( {\mathbb{D}} ; \emptyset , E_n)$ and the corresponding mapping
class in ${\mathfrak M} ( {\mathbb{D}} ; \emptyset , E_n)$. This mapping class
is denoted by ${\mathfrak m}_{\varphi}^{\rm free}$
\index{$\mathfrak{m}_{\varphi}^{\rm free}$} and is called the
free isotopy class 
of $\varphi$. \index{$\varphi$} We also call the elements of
${\mathfrak M} (\bar {\mathbb{D}} ;
\partial {\mathbb{D}} , E_n)$ the relative mapping classes.\index{relative mapping classes} The restriction map
$\varphi \to \varphi \mid {\mathbb{D}}$ defines a surjective mapping from
${\mathfrak M} (\bar D ;
\partial {\mathbb{D}} , E_n)$ to ${\mathfrak M} ({\mathbb{D}} ; \emptyset , E_n)$. Indeed,
each element of ${\mathfrak M} ({\mathbb{D}} ; \emptyset , E_n)$ contains
representatives which extend to the boundary $\partial {\mathbb{D}}$ as the
identity mapping on $\partial \mathbb{D}$. There is a short exact sequence
\begin{equation}\label{eq2.2}
0 \to \mathcal{K} \to {\mathfrak M} (\bar {\mathbb{D}} ; \partial {\mathbb{D}} , E_n) \to
{\mathfrak M} ({\mathbb{D}} ; \emptyset , E_n) \to 0 \, .
\end{equation}
An element ${\mathfrak m} \in {\mathfrak M} (\bar {\mathbb{D}} ; \partial {\mathbb{D}},
E_n)$ is in the kernel $\mathcal{K}$ iff each representative is \index{$\mathcal{K}$}
isotopic to the identity through homeomorphisms of $ {\mathbb{D}}$ which
fix
$E_n$ setwise (and, hence, pointwise).


Let $X$ be an oriented surface.
A Dehn twist \index{Dehn twist} about a simple closed curve $\gamma$
in an oriented surface $S$ is a mapping that is isotopic to the following one. Take a neighbourhood of
$\gamma$ that can be parameterized  as a round annulus  $A=\{ e^{-\varepsilon} < \vert z \vert < 1\}$  so that $\gamma$ corresponds to $|z|=e^{-\frac{\varepsilon}{2}}$. The mapping is
an orientation preserving self-homeomorphism of $S$
which is the identity outside $A$
and is equal to the mapping $e^{-\varepsilon s +2\pi i t}\to  e^{-\varepsilon s +2\pi i (t+s) }$ for  $e^{-\varepsilon s +2\pi i t}\in A$, i.e. $s\in (0,1)$. Here $\varepsilon$ is a small positive number.

For the following theorem see e.g. \cite{Joh}.
\begin{thm}\label{thm2.1} The kernel $\mathcal{K}$ is generated by a
Dehn twist about a simple closed curve $\gamma$ in ${\mathbb{D}}$ which is
homologous to $\partial {\mathbb{D}}$ in ${\mathbb{D}} \backslash E_n$.
\end{thm}
\index{$\mathcal{H}_{\infty}$} \index{$\mathfrak{M} ({\mathbb P}^1;
\infty, E_n)$}
Let again $E_n \subset {\mathbb D}$ be a set that
contains exactly $n$ points. The inclusion $\overline{\mathbb{D}} \hookrightarrow \mathbb{P}^1$ induces a surjective homomorphism
${\mathcal H}_{\infty}$
\begin{equation}
\label{eq2.7} {\mathcal H}_{\infty} : {\mathfrak M} (\bar{\mathbb
D};
\partial \, {\mathbb D} , E_n) \to {\mathfrak M} ({\mathbb P}^1;
\infty, E_n),
\end{equation}
 which is described in detail as
follows. Take a mapping class ${\mathfrak m} \in {\mathfrak
M}(\overline{\mathbb D};
\partial \, {\mathbb D} , E_n)$. Represent it by a homeomorphism
$\varphi$. Extend $\varphi$ to a self-homeomorphism
$\varphi_{\infty}$ of ${\mathbb P}^1$ by putting $\varphi_{\infty} =
\varphi$ on $\overline{\mathbb D}$ and $\varphi_{\infty} = {\rm id}$
outside $\overline{\mathbb D}$. Let ${\mathfrak m}_{\infty} \in
{\mathfrak M} ({\mathbb P}^1; \infty, E_n)$ be the mapping class of
$\varphi_{\infty}$. It depends only on the mapping class
$\mathfrak{m}$ of $\varphi$ in ${\mathfrak M} (\bar{\mathbb D};
\partial \, {\mathbb D} , E_n)$, not on the choice of the representative $\varphi$. Put
\begin{equation}\label{eq2.8}
{\mathcal H}_{\infty} ({\mathfrak m}) \stackrel{def}= {\mathfrak
m}_{\infty} \in {\mathfrak M} ({\mathbb P}^1;  \infty , E_n).
\end{equation}
\index{$\mathcal{H}_{\infty} ({\mathfrak m})$}
\index{$\mathfrak{m}_{\infty}$} \index{$\varphi$} \index{$\varphi_{\infty}$}
The homomorphism is surjective since each element of  $ {\mathfrak M} ({\mathbb P}^1;  \infty , E_n)$ can be represented by a homeomorphism that is equal to the identity on ${\mathbb P}^1\setminus \mathbb{D}$. By Theorem \ref{thm2.1} the kernel of the homomophism  \eqref{eq2.8} consists of the class of powers of Dehn twists about a simple closed curve that is homologous to $\partial \mathbb{D}$.

Let $A$ and $A'$ be oriented manifolds, let $A_1$ and $A_2$ be subsets of $A$, and let $A_1'$ and $A_2'$ be subsets of $A'$. Take a homeomorphism $\psi:A'\to A$ that maps $A'_j$ to $A_j$ for $j=1,2$.  The mapping $\varphi\to \psi^{-1} \circ \varphi\circ \psi$ is a one-to-one mapping from ${\rm Hom}^+(A;A_1,A_2)$ onto ${\rm Hom}^+(A';A'_1,A'_2)$. It defines an isomorphism ${\rm is}_ \psi$ from the mapping class
$\mathfrak{M}(A;A_1,A_2)$ onto the mapping class  $\mathfrak{M}(A';A'_1,A'_2)$. The isomorphism is determined by $\psi$ up to
conjugation by an element of  $\mathfrak{M}(A;A_1,A_2)$. Indeed, let $\psi_1:A'\to A$ be another homeomorphism that maps $A'_1$ to $A_1$ and $A'_2$ to $A_2$. Then
\begin{align*}
(\psi_1)^{-1}\circ\varphi\circ\psi_1= (\psi^{-1} \circ\psi_1)^{-1}\circ \psi^{-1}\circ \varphi \circ\psi\circ  (\psi^{-1} \circ\psi_1) \,.
\end{align*}
The conjugacy class $\widehat {\mathfrak{m}}$ of an element $\mathfrak{m}$ of the mapping class group $\mathfrak{M}(A;A_1,A_2)$ consists of all elements $\mathfrak{g}^{-1} \mathfrak{m}\mathfrak{g}$ with $\mathfrak{g}\in \mathfrak{M}(A;A_1,A_2)$.The set of conjugacy classes is denoted by $\widehat{\mathfrak{M}}(A;A_1,A_2)$. We will identify the conjugacy class of an element $\mathfrak{m}\in \mathfrak{M}(A;A_1,A_2)$    
with the set
\begin{align*}
\big\{\psi^{-1} \varphi \circ \psi: \, \varphi\in \mathfrak{m},\, \psi: & \, A'\to A\, \mbox{a homeomorphism that maps }\\
& A'_1 \,\mbox{to}\,A'_2\,
\mbox{ and}\, A'_2 \, \mbox{to}\, A_2\big\}\,,
\end{align*}
having in mind the isomorphisms ${\rm is}_{\psi}^{-1}$ that are determined up to conjugation by elements of $\mathfrak{M}(A;A_1,A_2)$   .
\index{conjugacy class of mappings}
\index{$\widehat {\mathfrak{m}}$}

Consider the set of conjugacy classes $   \widehat{\mathfrak M} ({\mathbb P}^1;  \infty , E_n)$. 
\index{$\widehat{\mathfrak M} ({\mathbb P}^1; \infty,E_n)$}
Any homeomorphism $h:\mathbb{D}\setminus E_n\to \mathbb{C}\setminus E_n$ that maps points close to $\partial \mathbb{D}$ to points close to $\infty$  defines an isomorphism
$${\mathfrak M} ({\mathbb D};  \emptyset , E_n) \to
 {\mathfrak M} ({\mathbb P}^1;  \infty , E_n)
$$
by associating to each orientation preserving self-homeomorphism $\varphi:\mathbb{D}\setminus E_n\toitself$ the self-homeomorphism $h\circ\varphi\circ h^{-1}$ of $\mathbb{C}\setminus E_n$. The isomorphisms corresponding to different homomorphisms $\mathbb{D}\setminus E_n\to \mathbb{C}\setminus E_n$
differ by conjugation by a self-homeomorphism of $\mathbb{C}\setminus E_n$ whose extension to $\mathbb{P}^1$ fixes $\infty$.
We obtain a one-to-one correspondence between the set of conjugacy classes $\widehat{\mathfrak{M}}(\mathbb{D}\setminus E_n)$ and the set of conjugacy classes
$  \widehat{\mathfrak M} ({\mathbb P}^1 \setminus E_n;  \infty) \cong  \widehat{\mathfrak M} ({\mathbb P}^1;  \infty , E_n)$.

More generally,  if $X$ is a closed oriented surface and $E_n$ a subset of $X$ consisting of $n$ points, the space ${\mathfrak M} (X ; \emptyset ,E_n) \cong {\mathfrak M} (X\backslash E_n)$ is called the mapping class group of the
$n$-punctured surface.
Let $X$ be a compact connected oriented surface with boundary and let $E_n$ be a subset of $X$ that contains
exactly $n$ points. The set ${\mathfrak M} (X ; \partial X$, $E_n)$ $\cong$
${\mathfrak M} (X \backslash E_n; \partial X)$ \index{$\mathfrak{M}
(X ; \partial X, E_n)$}
is called the mapping class group of the
$n$-punctured surface  with boundary.
Consider the mapping $\varphi\to \varphi|{\rm Int} X$ that assigns to each self-homeomorphism $\varphi$ of $X$ the restriction to the interior ${\rm Int}X$ of $X$. This mapping defines a homomorphism ${\mathfrak M} (X ; \partial X,\;  E_n) \to {\mathfrak M} ({\rm Int }X ; \emptyset,\;  E_n)$, whose kernel is generated by Dehn twists about closed curves in $X$ that are homologous to a boundary component of $X$.

Let $X$ be a bordered Riemann surface with smooth boundary. There is a compact Riemann
surface $X^c$ \index{$X^c$} and a diffeomorphism from $X$ onto the
closure of a smoothly bounded domain in $X^c$ with the following properties. The
diffeomorphism is conformal on ${\rm Int} \, X$.
The domain in $X^c$
is obtained by removing from $X^c$ a finite number of smoothly bounded topological discs. This follows from the fact, that the interior of $X$ admits a conformal mapping onto a  smoothly bounded domain on a closed Riemann surface
(\cite{Sto}, see also  Section \ref{sec:2.0}), and the smooth analog of Caratheodory's Theorem (Theorem 4 Chapter 2, Section 3, in \cite{Go}) that states smooth extension to the boundary of  conformal mappings between smoothly bounded domains (see e.g. \cite{Bel}).
Let $\partial_1 , \ldots ,
\partial_N$ be the boundary components of $X$ and let $\delta_1 ,
\ldots , \delta_N$ be the open discs on $X^c$ bounded by the
$\partial_j$. Let $E_n \subset {\rm Int} \, X$ be a finite set. For
each $j = 1 , \ldots , N$ we pick a point $\zeta_j \in \delta_j$.
Put $\zeta=\{\zeta_1 , \ldots , \zeta_N \}$.
There is a homomorphism 
\begin{equation}\label{eq2.10}
{\mathcal H}_{\zeta} : {\mathfrak M} (X; \partial X,E_n) \to
{\mathfrak M} (X^c ; \{\zeta_1 , \ldots , \zeta_N \}, E_n)
\end{equation}
which can be described as follows.
\index{$\mathfrak{M}(X^c ; \{\zeta_1 , \ldots , \zeta_N \}, E_n)$}
Take ${\mathfrak m} \in {\mathfrak M} (X;  \partial X , E_n)$. Let
$\varphi \in {\rm Hom}^+ (X;  \partial X , E_n)$ be a representing
homeomorphism. Let $\varphi^c$ be the extension of $\varphi$ to
$X^c$ which is the identity outside $X$. Then ${\mathcal H}_{\zeta}
({\mathfrak m}) = {\mathfrak m}_{\zeta}$ where ${\mathfrak
m}_{\zeta}$ is the class of $\varphi^c$ in ${\mathfrak M} (X^c ; \{
\zeta_1 , \ldots , \zeta_N \} , E_n)$. The homomorphism  \eqref{eq2.10} is surjective. Its kernel is generated by Dehn twists about closed curves that are homologous to a boundary component of $X$. \index{$\mathcal{H}_{\zeta}$} \index{$\mathfrak{m}_{\zeta}$}

We remove now the requirement that the elements of a mapping class  fix the boundary of a bordered Riemann surface $X$ pointwise (or fixes each boundary component setwise), in other words we consider mapping classes in
${\mathfrak M} (X; \emptyset,E_n)$. We obtain a homomorphism
$$
{\mathfrak M} ({\rm Int}X; \emptyset,E_n)\to  {\mathfrak M} (X^c\setminus  \{
\zeta_1 , \ldots , \zeta_N \} ;\emptyset,E_n)\,.
$$
We may identify the conjugacy classes
$$
\widehat{\mathfrak{M}}  ({\rm Int}X;\emptyset, E_n)\cong \widehat{\mathfrak M} (X^c \setminus \{
\zeta_1 , \ldots , \zeta_N \} ; \emptyset, E_n)\,.
$$

\smallskip

\section{The relation between Braids and Mapping Classes}
\label{sec:2.2a}
In this section we describe the isomorphism between the braid group
${\mathcal B}_n$ and the mapping class group of the $n$-punctured
disc. We use complex notation.


For any subset $E$ of the unit disc and any self-homeomorphism $\psi$ in 
${\rm Hom}^+(\overline{\mathbb D} , \emptyset, \emptyset)$ we
put ${\rm ev}_E \, \psi = \psi (E)$. For $E_n^0 = \left\{ 0,\frac1n ,
\ldots , \frac{n-1}n \right\}$ ( considered as unordered tuple of
$n$ points or as set) we obtain
\begin{equation}\label{eq2.10a}
{\rm ev}_{E_n^0} \, \psi = \left\{ \psi (0) , \psi \left(\frac1n
\right), \ldots , \psi \left( \frac{n-1}n \right) \right\} \, .
\end{equation}
We define the mapping $e_n$, \index{${\rm ev}_E \, \psi$}
\begin{equation}\label{eq2.10b}
e_n (\psi) = \left( \psi (0) , \ldots , \psi \left( \frac{n-1}n
\right) \right)
\end{equation}
\index{$e_n (\psi)$}that assigns to $\psi$ the ordered $n$-tuple of values of $\psi$ at the points
$0,\frac{1}{n},\ldots,\frac{n-1}{n}$. Note that
$e_n (\psi) \in  C_n ({\mathbb C})$, and ${\mathcal P}_{\rm sym} \, e_n
(\psi) = {\rm ev}_{E_n^0} \, \psi$. The isomorphism between the
mapping class group ${\mathfrak M} (\overline{\mathbb D}, \partial
\, {\mathbb D} , E_n^0)$ and the group of isotopy classes of
geometric braids with base point $E_n^0$ is obtained as follows.

Let $\varphi \in {\rm Hom}^+  (\overline{\mathbb D} ; \partial \,
{\mathbb D}, E_n^0)$. Consider a path $\varphi_t \in {\rm Hom}^+
(\overline{\mathbb D} , \partial \, {\mathbb D})$, $t \in [0,1]$,
which joins $\varphi$ with the identity. In other words, $\varphi_t$
is a continuous family of self-homeomorphisms of $\overline{\mathbb
D}$ which fix the boundary $\partial \, {\mathbb D}$ pointwise, such
that $\varphi_0 = {\rm id}$ and $\varphi_1 = \varphi$. Notice that we do not
require that $\varphi_t$ maps $E_n^0$ to itself. By the
Alexander-Tietze Theorem
(see e.g. \cite{KaTu}, Section 1.6.1)  such a family exists for each $\varphi \in
{\rm Hom}^+ (\overline{\mathbb D} ; \partial \, {\mathbb D} ,
E_n^0)$. Consider the evaluation map \index{evaluation map}
\begin{equation}\label{2.10g}
[0,1] \ni t \to {\rm ev}_{E_n^0} \, \varphi_t = \left\{ \varphi_t
(0) , \ldots , \varphi_t \left( \frac{n-1}n \right) \right\} \in C_n
({\mathbb C}) \diagup {\mathcal S}_n \, .
\end{equation}
\index{$\varphi_t$}
This map defines a geometric braid  in the cylinder $[0,1] \times {\mathbb
D}$ (i.e. ${\rm ev}_{E_n^0} \, \varphi_t \in C_n({\mathbb D})
\diagup {\mathcal S}_n$ for each $t$). Its base point is
$$
E_n^0 = \varphi_0 (E_n^0) = \varphi_1 (E_n^0) \, .
$$
Notice that the isotopy class of the obtained geometric braid
depends only on the class of $\varphi$ in ${\mathfrak M}
(\overline{\mathbb D} ; \partial \, {\mathbb D} , E_n^0)$. The
obtained mapping from ${\mathfrak M} (\overline{\mathbb D} ;
\partial \, {\mathbb D} , E_n^0)$ to the group of braids with base
point $E_n^0$, hence to ${\mathcal B}_n$, is a homomorphism. It is,
in fact, an isomorphism. This is a consequence of Proposition
\ref{prop2.2} below.

Let $E_n$ be an arbitrary unordered $n$-tuple of points in ${\mathbb
D}$. A continuous family of homeomorphisms $\varphi_t \in {\rm
Hom}^+ (\overline{\mathbb D} ; \partial \, {\mathbb D})$, $t \in
[0,1]$, is called a parameterizing isotopy of a \index{isotopy ! parameterizing
isotopy} geometric braid $f : [0,1] \to C_n ({\mathbb D}) \diagup
{\mathcal S}_n$ with base point $E_n$ if $\varphi_0 = {\rm id}$ and
${\rm ev}_{E_n} \, \varphi_t = f(t)$, $t \in [0,1]$.

\begin{prop}\label{prop2.2}
Let $f : [0,1] \to C_n ({\mathbb D})\diagup {\mathcal S}_n$ be a smooth
geometric braid in $[0,1] \times \overline{\mathbb D}$ with any base
point $E_n \in C_n ({\mathbb D}) \diagup {\mathcal S}_n$. Then there
exists a smooth parameterizing isotopy $\varphi_t$ for $f$.
The mappings $\varphi_t$ can be chosen so that the map $[0,1] \times
\overline{\mathbb D} \ni (t,z) \to (t,\varphi_t (z)) \in [0,1]
\times \overline{\mathbb D}$ is a diffeomorphism.
\end{prop}
For the continuous version see e.g. \cite{KaTu}.

\medskip

\noindent {\bf Proof of Proposition \ref{prop2.2}}.
Lift the mapping
$f$ to a mapping $\tilde f : [0,1] \to C_n ({\mathbb D})$. Denote
the coordinate functions of $\tilde f$ by $f_j : [0,1] \to {\mathbb
D}$. Choose $\delta > 0$ so that for each $t$ the discs ${\mathcal
U}_j (t)$ of radius $\delta$ around the points $f_j (t)$, $j =
1,\ldots , n$, are pairwise disjoint subsets of ${\mathbb D}$.

Consider for each $j$ the tube ${\mathfrak T}_j = \underset{t \in
[0,1]}{\bigcup} (t,{\mathcal U}_j(t))$. Define a smooth mapping $v$
from the union $\underset{j}{\bigcup} \ {\mathfrak T}_j$ of the tubes
to
the complex plane $\mathbb{C}$ by putting
\begin{equation}\label{eq2.10h}
v(t,\zeta)\stackrel{def}=\frac{\partial}{\partial t} f_j(t),\,\; (t,\zeta)\in {\mathfrak T}_j\,.
\end{equation}
Notice that for $(t,\zeta)\in {\mathfrak T}_j$ the point $\zeta$ is in the $\delta$-neighbourhood $\mathcal{U}_j(t)$ of $f_j(t)$ in $\mathbb{D}$. For each $t$ the mapping $v$ is constant on each $\mathcal{U}_j(t)$ as a function of $\zeta$.
The function $f_j$ satisfies the differential equation
\begin{equation}\label{eq2.10i}
\frac{\partial}{\partial t}f_j(t)= v(t,f_j(t)), \, t \in [0,1]\,.
\end{equation}
Indeed, $(t,f_j(t))$ is contained in ${\mathfrak T}_j$, hence,
$v(t,f_j(t))=\frac{\partial}{\partial t} f_j(t)$.

Let ${\mathfrak T}^0_j \Subset {\mathfrak T}_j$ be the $\frac{\delta}{2}$-tubes in $[0,1]\times \mathbb{D}$ around the graphs of $f_j$. Let $\chi_0$ be a $C^{\infty}$-function on $\mathbb{C}$ with values in $[0,1]$ that equals $1$ on $|\zeta|\leq \frac{\delta}{2}$ and equals $0$ outside $|\zeta|< \delta$.
Define a mapping $\chi$ on $[0,1]\times \mathbb{D}$ by the equations $\chi(t,\zeta)= \chi_0(\zeta- f_j(t))$ for $(t,\zeta)\in \mathfrak{T}_j$, $j=1,\ldots,n,$ and $\chi(t,\zeta)=0$ else.
Then $\chi$ is a $C^{\infty}$-function on $[0,1]\times \mathbb{D}$ with values in $[0,1]$ that equals $1$ in the union $\underset{j}{\bigcup} \ {\mathfrak T}^0_j$ of the smaller tubes and equals $0$ outside the union of the bigger tubes $\underset{j}{\bigcup} \ {\mathfrak T}_j$.
Put
\begin{equation}\label{eq2.10j}
V(t,\zeta)=\begin{cases} v(t,\zeta)\cdot \chi(t,\zeta) &\text{if} \;(t,\zeta) \in \underset{j}{\bigcup} \ {\mathfrak T}_j\\
0& \text{otherwise}
\end{cases}
\end{equation}
For $\zeta \in \mathbb{D}$ we denote by $\varphi_t(\zeta)$ the solution of the initial value problem
\begin{equation}\label{eq2.10j'}
\frac{\partial}{\partial t}\varphi_t(\zeta)= V(t,\varphi_t(\zeta)),\, \varphi_0(\zeta)=\zeta\,,
\end{equation}
on the maximal interval of existence. By the smooth version of the Picard-Lindel\"of Theorem the initial value problem has a unique solution on a maximal interval of existence and the solution depends smoothly on all parameters (see e.g. Hartmann, Theorem V, 3.1, Corollary V, 3.1 and the remark after it, and also Theorem V, 4.1).
Near the ends of the maximal interval of existence the solution curve approaches the boundary of the domain. Since $V=0$ in a neighbourhood of $[0,1]\times \partial \mathbb{D}$, the solution curve approaches $\{1\}\times \mathbb{D}$, in other words the maximal interval of existence equals $[0,1]$. Moreover, for each $t$ the mapping $\zeta\to \varphi_t(\zeta)$ is a local diffeomorphism, and hence by the uniqueness theorem for solutions of initial value problems, this mapping is a global diffeomorphism. Since $V=v$ on the union $\underset{j}{\bigcup} \ {\mathfrak T}^0_j$ of the smaller tubes, equation \eqref{eq2.10i} yields $\varphi_t(f_j(0))=f_j(t)$ for all $j$ and $t\in[0,1]$. Since $V$ vanishes in a neighbouhood of $[0,1]\times \overline{\mathbb{D}}$, $\varphi_t(\zeta)=\zeta$ for $\zeta$ close to  $\overline{\mathbb{D}}$. Hence,  $\varphi_t,\,t \in[0,1]$, is a parameterizing isotopy.
\hfill $\Box$

\begin{rem}\label{rem2.2}
Let $f:[0,1]\times[0,1]\to C_n ({\mathbb D})\diagup {\mathcal S}_n$
be a smooth isotopy of braids with
fixed base point $E_n$, that joins the braids $t\to f(t,0)$ and $t\to f(t,1)$.
Then there exists a diffeomorphism $[0,1]\times[0,1]\times \overline{\mathbb{D}}\ni(t,s,\zeta)   \overset{\varphi_{t,s}(\zeta) }\longrightarrow [0,1]\times[0,1]\times \overline{\mathbb{D}} $ such that
$\varphi_{t,s}(E_n)=f(t,s),\, (t,s)\in [0,1]\times[0,1],$ and $\varphi_{t,s}$ is the identity on $[0,1]\times[0,1]\times \partial{\mathbb{D}}$. If $ \varphi_t^0$ is a given smooth parameterizing isotopy for $t\to f(t,0)$ and $ \varphi_t^1$ is a given smooth parameterizing isotopy for $t\to f(t,1)$, then
the family $\varphi_{t,s}$ may be chosen so that  $ \varphi_{t,0}=\varphi_t^0$, $ \varphi_{t,1}=\varphi_t^1$, and
$\varphi_{0,s}={\rm Id}$.
\end{rem}
The proof follows along the same lines as the proof of Proposition \ref{prop2.2}. We construct a
vector field $V$ on $ [0,1]\times[0,1]\times \overline{\mathbb{D}} $,
i.e. a vector field that depends on  $(\zeta,t)\in \overline{\mathbb{D}}\times [0,1]$ and the additional parameter $s\in [0,1]$, such that
$$
V(t,s,f_j(t,s))= \frac{\partial}{\partial t}f_j(t,s)\; \mbox{for} \; (t,s)\in [0,1]\times [0,1]
$$
for the strands $f_j$ of $f$, and solve the initial value problem
\begin{align*}
\frac{\partial}{\partial t}\varphi_{t,s}(\zeta)= V(t,s,\varphi_{t,s}(\zeta)),\, \varphi_{0,s}(\zeta)=\zeta\,,(s,\zeta)\in [0,1]\times \overline{\mathbb{D}}\,.
\end{align*}
The families $ \varphi_{t,0}$ and
$ \varphi_{t,1}$ define vector fields  $\frac{\partial}{\partial t} \varphi_{t,0}$ and   $\frac{\partial}{\partial t} \varphi_{t,1}$      for $(t,s,\zeta)$ in $[0,1]\times \{0\}\times \overline{\mathbb{D}} $ and
$[0,1]\times \{1\}\times \overline{\mathbb{D}} $. The vector field $V$ can be chosen to coincide with these vector fields on the respective sets. We omit the details.
\smallskip

By the proposition the inverse of the mapping ${\mathfrak M}
(\overline{\mathbb D} ; \partial \, {\mathbb D} , E_n^0) \to
{\mathcal B}_n$ is obtained as follows. Take a braid $b \in
{\mathcal B}_n$ and choose a representing smooth geometric braid in $[0,1]
\times {\mathbb D}$. Consider a parameterizing isotopy $\varphi_t$.
Associate to $b$ the mapping class of the homeomorphism $\varphi_1$.


Explicitly, the inverse mapping assigns to each generator $\sigma_j
\in {\mathcal B}_n$ the class of the following homeomorphism which
is called a half-twist around the interval $\left[ \frac{j-1}{n} ,
\frac {j}{n} \right]$. Take two open discs $D_1$ and $D_2$ centered
at the midpoint of the segment $\left[ \frac{j-1}n , \frac jn
\right]$, such that $\left[ \frac{j-1}n , \frac jn \right] \subset
D_1 , \bar D_1 \subset D_2$, $\bar D_2$ does not contain points of
$E_n^0$ other than $\frac{j-1}n$ and $\frac jn$. Define $\varphi
_{\sigma_j}$ to be the identity on $\overline{\mathbb D} \backslash
D_2$ and to be counterclockwise rotation by the angle $\pi$ on $\bar
D_1$. Extend this mapping by a homeomorphism of $\bar D_2 \backslash
D_1$ which changes the argument of each point by a non-negative
value at most equal to $\pi$. We will denote the mapping
$\mathcal{B}_n \to {\mathfrak M} (\overline{\mathbb D} ;
\partial \, {\mathbb D} , E_n^0)$ by $\Theta_n$. \index{$\Theta_n$}

\begin{figure}[H]
\begin{center}
\includegraphics[width=120mm]{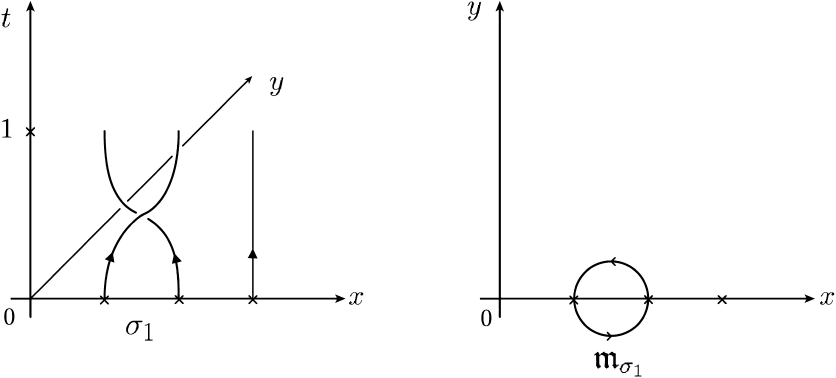}
\end{center}
\caption{The braid $\sigma_1\in \mathcal{B}_3$ and the corresponding mapping class $\mathfrak{m}_{\sigma_1}$}\label{fig2.1}
\end{figure}
The mapping class $\mathfrak{m}_{\Delta_n^2}$ corresponding to $\Delta_n^2$ is the Dehn twist about a curve in $\mathbb{D}$ that is homologous to $\partial \mathbb{D}$ in $\mathbb{D}\setminus E_n^0$.
Notice that the elements of the set of conjugacy classes $\widehat {\mathcal{B}}_n$ of braids are in
one-to-one correspondence with the elements of the set of conjugacy classes $\widehat{\mathfrak M}
(\overline{\mathbb D} ; \partial \, {\mathbb D} , E_n^0)$ of mapping classes.

\section{Beltrami differentials and quadratic differentials.}
\label{sec2:3}
In this section we collect some prerequisites from Teichm\"uller theory, in particular we recall some facts related to Beltrami differentials and quadratic differentials.
For more details we refer the reader for example to \cite{A1}, \cite{Let}, \cite{Na}, \cite{Str}.

Let $X$ be a Riemann surface. A Beltrami differential
\index{Beltrami differential} $\mu$ \index{$\mu$} on $X$ assigns to
each chart on $X$ with holomorphic coordinates $z$ an essentially
bounded measurable function $\mu (z)$ so that $\mu (z) \,
\frac{\overline{dz}}{dz}$ \index{$\mu (z) \,
\frac{\overline{dz}}{dz}$} is invariant under holomorphic coordinate
changes. In other words, the functions $\mu_1 (z)$ and $\mu_2
(\zeta)$ associated to local coordinates $z=z(\zeta)$ and  $\zeta$ are
related by the equation
\begin{equation}\label{eq2.14}
\mu_1 (z(\zeta)) \left( \frac{dz}{d\zeta} \right)^{-1}
\overline{\left( \frac{dz}{d\zeta} \right)} = \mu_2 (\zeta) \, .
\end{equation}
We may interpret Beltrami differentials as sections in the
bundle $\kappa^{-1}$ $\otimes$ $\bar\kappa$ for the cotangent bundle
$\kappa$ of $X$.


By an abuse of notation we will denote the Beltrami differential on
$X$ by $\mu$ and write $\mu = \mu (z) \, \frac{d\bar z}{dz}$, where
the left hand side denotes the globally defined Beltrami
differential and $\mu (z)$ on the right hand side is a representing
function in local coordinates $z$. The value $\vert \mu (z) \vert$
is invariant under holomorphic coordinate change. Put $\Vert \mu
\Vert_{\infty} \overset{\rm def}{=} \underset{X}{\rm sup} \, \vert
\mu \vert$. \index{$\underset{X}{\rm sup} \, \vert \mu \vert$}


Let $X$ and $Y$ be Riemann surfaces and let $\varphi : X \to Y$ be a
smooth orientation preserving homeomorphism. Let $z$ be coordinates
near a point of $X$ and let $\zeta$ be coordinates on $Y$ near its
image under $\varphi$. We write $\zeta = \varphi(z)$ in these
coordinates. Consider the function $\frac{\bar\partial \,
\varphi}{\partial \, \varphi} = \frac{\frac{\partial}{\partial \,
\bar z} \, \varphi}{\frac{\partial}{\partial z} \, \varphi}$ in
these coordinates. Since for the Jacobian ${\mathcal J} (z)$ of
$\varphi (z)$, ${\mathcal J}^2 (z) = \left\vert
\frac{\partial}{\partial z} \, \varphi (z) \right\vert^2 -
\left\vert \frac{\partial}{\partial \, \bar z} \, \varphi (z)
\right\vert^2 > 0$, the denominator $\partial \varphi$ does not
vanish and $ \left\vert \frac{\bar\partial \, \varphi}{\partial \,
\varphi} \right\vert < 1$. The expression
$\frac{\frac{\partial}{\partial \, \bar z} \, \varphi (z) \, d \bar
z}{\frac{\partial}{\partial z} \, \varphi (z) \, dz}$ is invariant
under holomorphic coordinate changes on $X$ and on $Y$. It defines a
Beltrami differential $\mu_{\varphi}$ \index{$\mu_{\varphi}$} on
$X$. The mapping $\varphi$ is called quasiconformal \index{mapping !
quasiconformal} if $\Vert \mu_{\varphi} \Vert_{\infty} < 1$. If $X$
and $Y$ are compact this is automatically so. The condition that
$\varphi$ is differentiable can be weakened.

For the Beltrami differential of superpositions of quasiconformal mappings the following formulas hold.
\begin{align}\label{eq2.14'}
\mu_g\circ f=\frac{\partial _z f}{ \partial _{\overline z} \overline{f}}
\frac{\mu_{g\circ f} -\mu_f  }{1-\overline{\mu_f} \, \mu_{g\circ f}}\,.
\end{align}
If $g$ is conformal, then $\mu_g=0$ and $\mu_{g\circ f}=\mu_f$. If
$f$ is conformal, then $\mu_f=0$ and
\begin{align}\label{eq2.14''}
\mu_g\circ f =\big( \frac{f'}{|f'|}\big)^2\, \mu_{g\circ f}\,.
\end{align}
For a detailed account
on quasiconformal mappings we refer to \cite{A1}.


The quasiconformal dilatation \index{quasiconformal dilatation} of
the mapping $\varphi$ is defined as $K(\varphi) = \frac{1+\Vert
\mu_{\varphi} \Vert_{\infty}}{1-\Vert\mu_{\varphi} \Vert_{\infty}}$
\index{$K(\varphi)$} if $\varphi$ is quasiconformal and is defined to be equal to $\infty$
otherwise.


A meromorphic (respectively, holomorphic) quadratic differential
\index{quadratic differential ! holomorphic} \index{quadratic
differential ! meromorphic} $\phi$ \index{$\phi$} on $X$ assigns to
each chart on $X$ with holomorphic coordinates $z$ a meromorphic
(respectively, holomorphic) function which by abusing notation we
denote by $\phi (z)$, such that $\phi (z)(dz)^2$ \index{$\phi
(z)(dz)^2$} is invariant under holomorphic changes of coordinates.
In other words, the functions $\phi_1 (z)$ and $\phi_2 (\zeta)$
associated to local coordinates $\zeta$ and $z(\zeta)$ are related
by the transformation rule
\begin{equation}\label{eq2.15}
\phi_1 (z(\zeta)) \left( \frac{dz}{d\zeta} \right)^2 = \phi_2
(\zeta) \, .
\end{equation}
Holomorphic quadratic differentials on $X$ can be regarded as
holomorphic sections of the bundle $\kappa^2$. By an abuse of
notation we will write $\phi = \phi (z) \, dz^2$, where the left
hand side denotes the quadratic differential and $\phi (z)$ on the
right hand side denotes the meromorphic function that represents
$\phi$ in local coordinates $z$. If the function $\phi_1(\zeta)$ corresponding to local coordinates $\zeta$ vanishes or equals $\infty$ at a point $\zeta_0$
the function $\phi_2(z)$ corresponding to local coordinates $z(\zeta)$  vanishes at the point $z(\zeta_0)$ or equals $\infty$ there. The set of points on $X$ where this happens is called the set of singular points of the quadratic differential $\phi$. The other points are called regular points of the quadratic differential.
\index{quadratic differential ! singular point of} \index{quadratic differential ! regular point of}

Suppose on a simply connected domain $U$ with coordinates $\zeta$ the quadratic differential equals $\phi(\zeta)(d\zeta)^2$ where $\phi \neq 0$ is a nowhere vanishing analytic function (in particular $\phi$ has no pole on $U$). Consider the differential form
$\sqrt{\phi(\zeta)}d\zeta$ on $U$ for a branch of the square root of $\phi$. There are local holomorphic coordinates $z$ of the quadratic differential
$\phi (\zeta) \, d\zeta^2$ on $U$ in which the quadratic differential
has the form $dz^2$. They are called 
flat coordinates. Such coordinates can be obtained as the integral of the mentioned differential form, \index{quadratic differential ! flat coordinates of}
\begin{equation}\label{eq2.15a}
z(\zeta)= \int_{\zeta_0}^{\zeta}\sqrt{\phi(\zeta')}d\zeta'.
\end{equation}
Together with the complex differential form $\sqrt{\phi(\zeta)}d\zeta$
on $U$ we will consider the two real forms $dx(\zeta)= \mbox{Re} (\sqrt{\phi(\zeta)}d\zeta)$ and $dy(\zeta)= \mbox{Im} (\sqrt{\phi(\zeta)}d\zeta)$ on $U$.

The coordinates \eqref{eq2.15a} vanish at $\zeta_0$. Flat coordinates that vanish at the point $\zeta_0$ are called distinguished coordinates at $\zeta_0$. 
Different flat coordinates on $U$ differ from each other
by sign (since we may choose the other branch of the square root) and an additive constant (as the integral of the form is defined up to an additive constant). Denote by $E'$ the set of singular points of the quadratic differential on $X$.
The $\phi$-regular set $X\setminus E'$ has a flat structure. This means that there is an atlas with transition functions being $z\to \pm z +  b$ for complex numbers $b$.

The flat structure defines foliations on the  $\phi$-regular part of $X$ as follows.
Cover $X\setminus E'$ by simply connected open sets $U_j$ each of which is equipped with distinguished coordinates $z_j$.
Then the forms $dx_j$ on $U_j$ define a foliation on $X\setminus E'$ which is called horizontal foliation. The plaques (the local parts of leaves in $U_j$ ) are the connected components of the sets $x_j=\mbox{const}$ in $U_j$. The transition functions map plaques to plaques.

\index{foliation ! horizontal} \index{foliation ! vertical}
In the same way the forms $dy_j$ on $U_j$ define a foliation on $X \setminus E'$ which is called the vertical foliation. Actually, for each real $\theta$ the forms $\mbox{Re}(e^{i\theta}dz_j)$ define a foliation on $X\setminus E'$.

The horizontal foliation on each $U_j$ is generated by a vector field $v_h(\zeta), \zeta \in U_j$. In other words, the leaves of this foliation are the integral curves of this vector field. In arbitrary coordinates $\zeta$ on $U_j$ the vector field is defined by the conditions
\begin{align}\label{eq2.300}
 \mbox{Re}(\sqrt{\phi(\zeta)}\;d\zeta)(v_h(\zeta))=1\,, \;\;  \mbox{Im}(\sqrt{\phi(\zeta)}\;d\zeta)(v_h(\zeta))=0\,.
\end{align}
Notice that if $\phi(\zeta)= |\phi(\zeta)| e^{i\theta}$ in coordinates $\zeta=\xi+i \eta$,  then in these coordinates
\begin{align*}
\mbox{Re}(\sqrt{\phi(\zeta)}\;d\zeta)= \sqrt{ |\phi(\zeta)|}(cos{\frac{\theta}{2}}\,d\xi -\sin{\frac{\theta}{2}}\,d\eta)\,,\nonumber\\
 \mbox{Im}(\sqrt{\phi(\zeta)}\;d\zeta)= \sqrt{ |\phi(\zeta)|}(\sin{\frac{\theta}{2}}\,d\xi +\cos{\frac{\theta}{2}}\,d\eta)\,.
\end{align*}
In distinguished coordinates $z_j$ the conditions \eqref{eq2.300} become $dx_j(v_h)=1\,\,,\,$ $dy_j(v_h)=0\,.\,$ In other words, in these coordinates the vector field points in the direction of the positive real axis (or in the direction of the negative real axis if the other branch of the square root is chosen). Identifying a vector with a point in the complex plane we obtain the following.  If $\phi(\zeta)= |\phi(\zeta)| e^{i\theta}$ in coordinates $\zeta$ then in these coordinates $v_h(\zeta)= |\phi(\zeta)|^{-\frac{1}{2}}\binom{\;\;\;\cos \frac{\theta}{2}}{ -\sin\frac{\theta}{2}}$
which we identify with the point $|\phi(\zeta)|^{-\frac{1}{2}} e^{-\frac{1}{2}i\theta}$ in the complex plane. Moreover, $v_{v}(\zeta)=|\phi(\zeta)|^{-\frac{1}{2}}\binom{\sin \frac{\theta}{2}}{ \cos\frac{\theta}{2}}=  |\phi(\zeta)|^{-\frac{1}{2}}\binom{\;\;\;\cos \frac{\theta-\pi}{2}}{ -\sin\frac{\theta-\pi}{2}}$ which can be identified with the point $|\phi(\zeta)|^{-\frac{1}{2}} e^{-\frac{1}{2}i(\theta-\pi)}$ in the complex plane.

In general the $v_{h}$ (the $v_v$, repectively) in the coordinate patches do not define a global vector field on $X\setminus E'$ but merely a global line field. \index{line field} Indeed, the horizontal foliation is in general not orientable.
The leaves of the horizontal foliation (of the vertical foliation, respectively) are also called horizontal trajectories (vertical trajectories, respectively). Arcs that are contained in horizontal (vertical, respectively) trajectories are called horizontal (vertical, respectively) arcs.

 \index{trajectory ! horizontal}
\index{trajectory ! vertical}

In a neighbourhood $U$ of a singular point $p_0$ there are holomorphic coordinates $z$
vanishing at the point in which the quadratic differential has the
form
\begin{equation}\label{eq2.16}
\phi (z) \, dz^2 = \left( \frac{a+2}2 \right)^2 z^a \, dz^2
\end{equation}
for some integer $a$. The coordinates in which $\phi$ has the form \eqref{eq2.16} are uniquely
defined up to multiplication by an $(a+2)$-nd root of unity and are
called distinguished coordinates at the singular points. The number $a$ is called the order
\index{quadratic differential ! order of a singular point} of the point. For distinguished
coordinates near regular points $\phi$ has the form \eqref{eq2.16} with
$a=0$. We will consider quadratic differentials with at worst poles
of order one, i.e. $a \geq -1$.

Equip a simply connected subset $U'$ of $U \setminus \{p_0\}$ with  coordinates $z$ satisfying \eqref{eq2.16}. We assume that in these coordinates
$U$ has the form $U=\{|z|<r_0\}$ for a positive number $r_0$.

The vector field that equals
$\frac{2}{a+2} r^{-\frac{a}{2}}e^{-\frac{i a \theta}{2}}$
at the point $z= e^{i\theta}$ in $U'$ in the mentioned coordinates generates the horizontal foliation on $U'$.
A ray $\{re^{i \theta}: 0<r<r_0\}\,$ (in these coordinates) is a subset of a horizontal trajectory if at each point the unit tangent vector to the ray is equal to the unit vector in a horizontal direction, i.e
$e^{i \theta}=\pm e^{-\frac{i a \theta}{2}}$, or equivalently, $\theta = \frac{2k\pi}{a+2}$, $k=0,\ldots,a+1$. The rays that are subsets of the vertical trajectories
are $\{re^{\frac{(2 k+1)\pi i}{a+2}}: r>0 \;\mbox{small}\},$ for $k=0,\ldots,a+1$ (we use the same coordinates satisfying \eqref{eq2.16}).
We call these rays horizontal bisectrices (also horizontal separatrices), or vertical bisectrices (also vertical separatrices), respectively.
Any other horizontal leaf in $U$ is a subset of a sector between two neighbouring horizontal bisectrices and avoids a neighbourhood (depending on the leaf) of the singular point. The respective facts hold for the vertical leaves.

The same facts hold in a neighbourhood of each singular point. We obtained a singular foliation on $X$. \index{singular foliations}
On a sector between two consecutive horizontal bisectrices a branch of the power $z \to z^{\frac{a+2}{2}}$ provides flat coordinates. The image of the sector under the latter mapping is a subset of the upper or lower half-plane with horizontal leaves
being the sets $x = \mbox{const}$. A sector between two consecutive vertical bisectrices is mapped to a subset of the right or left half-plane
with vertical leaves being the sets $y = \mbox{const}$.

We want to associate a singular metric to $\phi$.
Consider 
again a simply connected domain $U$ in $X\setminus E'$ equipped with coordinates $\zeta$.  The quadratic differential in these coordinates  has the form $\phi(\zeta)d\zeta^2$. Let $z$ be flat coordinates on $U$.
Let $\gamma:[0,1] \to U$ be a piecewise smooth curve.
We define the length $\ell_{\phi}(\gamma)$ with respect to $\phi$  of the curve $\gamma$ as its Euclidean length in some flat coordinates. This value is independent on the choice of the flat coordinates, and therefore it is well defined. If we denote by $\gamma_1$ the curve $t \to z(\gamma(t))$ in some flat coordinates $z$ then \index{$\ell_{\phi}(\gamma)$}
\begin{equation}\label{eq3.107b}
\ell_{\phi}(\gamma)= \int _{\gamma_1} |dz|= \int_0^1 \Bigl|\frac{d}{dt}(\gamma_1 (t)\Bigl|dt\,.
\end{equation}
In arbitrary coordinates $\zeta$ we obtain with $dz=\sqrt{\phi(\zeta)}d\zeta$
\begin{equation}\label{eq3.107c}
\ell_{\phi}(\gamma)= \int_{\gamma} |\sqrt{\phi(\zeta)}d\zeta|=
 \int_0^1 \sqrt{|\phi(\gamma(t))|}\; \Bigl|\frac{d}{dt}(\gamma (t))\Bigl|dt.
\end{equation}

We call  $|dz|=|\sqrt{\phi(\zeta)}d\zeta|$ the length element in the $\phi$-metric and $\ell_{\phi}(\gamma)$ the $\phi$-length of the curve $\gamma$. An arbitrary piecewise smooth curve in the regular part of $\phi$ can be subdivided into pieces contained in small discs around regular points. The \index{$\phi$-metric} \index{$\phi$-length}
$\phi$-length of the curve is the sum of the $\phi$-length of these pieces. The invariant form \eqref{eq3.107c} makes sense also for curves passing through singular points. As the singular points of $\phi$ are at worst poles of order one, the $\phi$-length of any piecewise smooth compact curve $\gamma: [0,1] \to X$ is finite.

In the same way we define the horizontal $\phi$-variation \index{$\phi$-variation ! horizontal} \index{$\phi$-variation ! vertical}
$\ell_{\phi,h}(\gamma)$ of a curve $\gamma$ (the vertical $\phi$-variation
$\ell_{\phi,v}(\gamma)$, respectively).  \index{$\ell_{\phi,h}(\gamma)$}  \index{$\ell_{\phi,v}(\gamma)$}    For a curve $\gamma$ in a simply connected set $U$ with
flat coordinates $z$  this is the total variation of the $x$ coordinate (the $y$ coordinate, respectively) along the curve $\gamma$,
\begin{align}\label{eq107d}
\ell_{\phi,h}(\gamma)= & \int _{\gamma} |dx|= \int_0^1 \Bigl|\mbox{Re}\frac{d}{dt}(\gamma (t)\Bigl|dt\,,\nonumber\\
\ell_{\phi,v}(\gamma)= & \int _{\gamma} |dy|= \int_0^1 \Bigl|\mbox{Im}\frac{d}{dt}(\gamma (t)\Bigl|dt\,.
\end{align}
An arbitrary curve is divided into pieces contained in such sets and the variations of the pieces are added.

The area element in the $\phi$-metric is defined as $dz\wedge d\bar z$ in flat coordinates $z$, or, equivalently, $\sqrt{\phi(\zeta)}d\zeta \wedge \overline{\sqrt{\phi(\zeta)}d\zeta}= |\phi(\zeta)| d\zeta\wedge \overline{ d\zeta}$ in arbitrary local coordinates $\zeta$.
For
a small open subset $U \subset X$ the integral $\underset{U}{\iint} \vert \phi \vert = \frac12 \underset{U}{\iint}
\vert \phi (\zeta) \vert \, d\zeta \wedge d\bar \zeta$ is invariant under
holomorphic coordinate changes. Hence, $\Vert \phi \Vert_1
\overset{\rm def}{=} \underset{X}{\iint} \vert \phi \vert$ is well
defined. $\phi$ is called integrable on $X$ if this integral is
finite. \index{quadratic differential ! integrable} \index{$\Vert \phi \Vert_1$}
The singularities of integrable meromorphic quadratic
differentials \index{quadratic differential ! integrable}
are zeros or
simple poles. The set of integrable meromorphic quadratic
differentials on $X$ with norm $\Vert \cdot \Vert_1$ forms a complex
Banach space.


The following fact is an immediate corollary of the definitions. For
each meromorphic quadratic differential $\phi$ and any number $k \in
(0,1)$ the object $k \, \frac{\vert \phi \vert}{\phi} = k \,
\frac{\bar\phi}{\vert \phi \vert}$ (given in local coordinates $z$
by $k \, \frac{\bar\phi (z)}{\vert \phi (z) \vert}$ with a
meromorphic function $\phi (z)$) is a Beltrami differential of norm
$K = \frac{1+k}{1-k}$.

\section{Elements of Teichm\"uller Theory.}
\label{sec:2.3a}

\noindent {\bf Teichm\"uller Theorem (closed surfaces).}
The following celebrated theorem was the key point for solving Riemann's problem on describing the classes of conformally equivalent Riemann surfaces. It was Teichm\"uller's vision that isotopy classes of conformally equivalent Riemann surfaces are easier to handle and their description provides a solution of Riemann's problem.
For more
details, see \cite{A2}, \cite{B2}.

\index{Teichm\"{u}ller ! Theorem}
\begin{thm}\label{thm2.3} Let $X$ and $Y$ be closed Riemann surfaces of
genus $g \geq 2$, and let $\varphi : X \to Y$ be a homeomorphism.
Then there is a unique homeomorphism isotopic to $\varphi$ with
smallest quasiconformal dilatation. This homeomorphism is either
conformal or its Beltrami differential has the form $k \cdot
\frac{\vert \phi \vert}\phi$ for a holomorphic quadratic
differential $\phi$ on $X$ and a constant $k \in (0,1)$. $\phi$ is
unique up to multiplication by a positive constant.
\end{thm}

A homeomorphism with Beltrami differential $k \cdot
\frac{\vert \phi \vert}\phi$ for a holomorphic quadratic
differential $\phi$ on $X$ and a constant $k \in (0,1)$
is called a Teichm\"uller
mapping and $\phi$ is called its quadratic differential. \index{Teichm\"uller ! mapping}

Teichm\"uller mappings are characterized by their local description in distinguished coordinates.
Let $X$ and $Y$ be closed Riemann surfaces.
Suppose $\varphi : X \to Y$ is a Teichm\"uller mapping
with quadratic differential $\phi$ and constant $k$. Then the
inverse mapping $\varphi^{-1}$ is again a Teichm\"uller mapping with
quadratic differential denoted by $-\psi$ and constant $k$. The
order of $\phi$ at a point $z$ is the same as the order of $-\psi$
at the image $\varphi (z)$. The quadratic differential $\psi$ is called the terminal quadratic differential of $\varphi$. The quadratic differential $\phi$ of $\varphi (z)$ is also called the initial quadratic differential of $\varphi (z)$. For more insight see
\cite{Be1}, or Theorem 8.1, V.8, of \cite{Let}. \index{quadratic differential ! initial} \index{quadratic differential ! terminal}


There are distinguished coordinates \index{quadratic differential !
distinguished coordinates of} $z$ for $\phi$ near a point $z_0 \in
X$ that vanish at $z_0$ and distinguished coordinates $\zeta$ for
$-\psi$ near the image point $\varphi (z_0) \in Y$ which vanish at
$\varphi (z_0)$ so that the mapping $\varphi$ has the form
\begin{equation}\label{eq2.17}
\zeta = \left( \frac{z^{a+2} + 2 \, k \vert z \vert^{a+2} + k^2 \,
\bar z^{a+2}}{1-k^2} \right)^{\frac1{a+2}}
\end{equation}
with $\zeta > 0$ for $z > 0$. If $a=0$ this is equivalent to
\begin{equation}\label{eq2.17a'}
\zeta
= \xi + i\eta = K^{\frac12} x + i \, K^{-\frac12} y
\end{equation}
with $z=x+iy$ and $K = \frac{1+k}{1-k}$ being the quasiconformal dilatation.


We will describe now an idea of Ahlfors ( \cite{A2}, section 4, p. 19-20) which allows to reduce
questions concerning self-homeomorphisms of punctured Riemann surfaces (self-homeomorphisms of closed Riemann surfaces with distinguished points, respectively) to the related
questions concerning closed Riemann surfaces. Ahlfors used it for a
proof of Teichm\"uller's theorem for punctured Riemann surfaces. It
can also be used for reducing the study of the entropy \index{entropy} of self-homeomorphisms of closed Riemann surfaces
with distinguished points
to the study of the entropy of self-homeomorphisms of closed Riemann surfaces.


Ahlfors' idea is the following. Let $X$ and $Y$ be closed Riemann surfaces, both with a set of $m$
distinguished points. (Equivalently, we may think about two $m$-punctured Riemann surfaces by using the identification of ${\rm Hom}^+
(X,\emptyset , E)$ with ${\rm Hom}^+ (X \backslash E,\emptyset ,
\emptyset )$.)
Assume $2g-2 + m > 0$, so that the universal covering of the
$m$-punctured surfaces equals ${\mathbb C}_+$. Except for $g=0$ and $m\leq 4$ Ahlfors associates to $X$ and
$Y$ closed Riemann surfaces $\hat X$ and $\hat Y$, which are
holomorphic simple branched coverings $p_X:\hat{X} \to X$ of $X$ and $p_Y:\hat{Y} \to Y$ of  $Y$ with branch
locus being the set of distinguished points, and have genus at least
two.
\index{covering ! simple branched}

Such coverings can be obtained as follows.
If the number of punctures is even (and not zero) and
either $g \geq 1$, or $g=0$ and $m \geq 6$, then one
can take a double branched covering with branch locus being the set of
distinguished points.

If $g>0$ and the number of points  in the set $E$ of distinguished points is odd and bigger than $1$, one first considers an intermediate double branched covering with branch locus  being a non-empty set $E' \subset E$ that contains an even number of points,and takes then a
double branched cover of the intermediate branched covering, this time with branch locus being the preimage of $E\setminus E'$  in the  intermediate branched covering (which consists of an even number of points).

If the number of distinguished points equals $1$ and $g>0$, we take an intermediate unramified covering and proceed further as before.

If $g=0$ and $m\geq 5$ is odd we consider an intermediate double branched covering with branch locus  being $m-1$ of the punctures, and then a double branched covering over the immediate covering branched at the preimages of the remaining puncture of $X$.
This treats all cases except $g=0$ and $m\leq 4$. If $m\leq 3$ there exist holomorphic self-mappings of $\mathbb{P}^1$ that map each distinguished point in $X$ to the respective distinguished point in $Y$. If $m=4$, the double branched covering with the required branch locus is a torus, in which case the  extremal mappings are well described, and also problems related to entropy are well-known in this case.
\index{covering ! branched}  \index{covering ! double branched} \index{covering ! unramified}

Consider a homeomorphism $\varphi : X \to Y$  which maps the set of
distinguished points of $X$ to the set of distinguished points of
$Y$. There is a lift $\;\hat\varphi \,$ of $\varphi\;$, $\hat\varphi :
\hat X \to \hat Y, \;$ i.e. a homeomorphism between closed
surfaces such that $p_Y \circ \hat \varphi = \varphi \circ p_X$.
The homeomorphism $\hat{\varphi}$ has some ``additional symmetries''.
Vice versa, a
homeomorphism between $\hat X$ and $\hat Y$ with such symmetries
is a lift of a homeomorphism between $X$ and $Y$ which maps the set
of distinguished points in $X$ to the set of distinguished points in
$Y$.

Notice that a self-homeomorphism $\varphi$ of a Riemann surface $X$ which preserves a set $E$ consisting of an even number of points, lifts to the double branched covering with branch locus $E$ as follows. Cut $X$ along a set $\Gamma$ of disjoint arcs that join pairs of points of $E$. Consider two copies of $X\setminus \Gamma$. $\varphi$ maps each copy to a copy of $X\setminus \varphi(\Gamma)$. Glue the copies in the source and the copies in the target crosswise along respective strands of the cuts and extend the mapping.

Teichm\"{u}ller's theorem also applies in the situation of
homeomorphisms with ``additional symmetry'' between closed Riemann
surfaces. One obtains quadratic differentials on $\hat X$ with a
"symmetry", and one obtains quadratic differentials on $X$ which
lift to the mentioned quadratic differentials on $\hat X$. This
implies that the quadratic differentials on $X$ have at most simple
poles at the branch locus. This is a simple calculation using the behaviour
of quadratic differentials under coordinate changes. Indeed, let $\hat z_0$ be a branch point in $\hat X$ and let $z_0$ be its image under the projection $\hat X \to X$. Let $\zeta$ be distinguished coordinates on $\hat X$ at $\hat z_0$, and let $z=z(\zeta)= \zeta ^2$ be coordinates near $z_0$ on $X$. If $\phi$ is the quadratic differential on $X$ and $\hat \phi$ its lift to $\hat X$ then by \eqref{eq2.15}
\begin{equation}\label{eq2.17a}
\hat{\phi}(\zeta)=\phi(z(\zeta)) z'(\zeta)^2=
\phi(z(\zeta)) (2\zeta)^2= 4 \phi(z(\zeta))z(\zeta)= 4 \phi(z) z \,.
\end{equation}
Since $\hat\phi$ does not have poles,
$\phi$ has at most a simple pole at $z_0$.

The arguments imply that
Teichm\"uller's theorem remains true for homeomorphisms between
closed  surfaces with distinguished points, if instead of
holomorphic quadratic differentials one considers meromorphic
quadratic differentials on closed surfaces with at most simple poles
at distinguished points. These are exactly the integrable
meromorphic quadratic differentials.


In a similar way Ahlfors treats finite Riemann surfaces of the second kind (i.e. Riemann
surfaces with finitely many boundary components, some of which may
be points, some are continua). He uses an extension of the
homeomorphism to a homeomorphism between the doubles of the Riemann
surfaces. We do not need this case here. For details see \cite{A2}.

The local characterization of Teichm\"uller maps implies the following useful fact. Let $X$, $X_1$ and $X_2$ be closed or punctured Riemann surfaces, and let $\varphi_1:X\to X_1$ and  $\varphi_2:X_1\to X_2$ be Teichm\"uller maps with quadratic initial differentials $\phi_1$ and $\phi_2$, respectively. Suppose $\phi_2$ is equal to the terminal quadratic differential of $\varphi_1$. Then $\varphi_2\circ\varphi_1$ is a Teichm\"uller map with quadratic differential $\phi_1$.

\bigskip

\noindent {\bf Teichm\"uller spaces.}
Let $X$ be a connected Riemann
surface of genus $g$, with $\ell$ boundary continua and $m$
punctures. We always assume that the universal covering of $X$
equals ${\mathbb C}_+$. (This excludes the Riemann sphere, ${\mathbb
C}$, ${\mathbb C}^* = {\mathbb C} \backslash \{ 0 \}$, and tori
(compact Riemann surfaces of genus 1).) \index{Teichm\"{u}ller ! space}

Let $w_j : X \to Y_j$, $j=1,2$, be two quasiconformal homeomorphisms
onto Riemann surfaces (considered as conformal structures on $X$). \index{conformal structure}
They are called (Teichm\"uller) equivalent if there exists a
conformal mapping $c : Y_1 \to Y_2$ such that $w_2^{-1} \circ c
\circ w_1$ is isotopic to the identity by an isotopy which fixes the
set of boundary continua pointwise (if any). We denote the
equivalence class of a quasiconformal complex structure $w : X \to
Y$ on $X$ by $[w]$. \index{$[w]$} The set of equivalence classes is
the Teichm\"uller space ${\mathcal T} (X)$.  \index{$\mathcal {T} (X)$}
Equip the Teichm\"uller
space with the Teichm\"uller metric \index{Teichm\"{u}ller ! metric}
\index{metric ! Teichm\"{u}ller} $d_{\mathcal T}$,
\begin{equation}\label{eq2.18}
d_{\mathcal T} ([w_1] , [w_2]) \overset{\rm def}{=} {\rm inf}
\left\{ \frac12 \log K (v_2 \circ v_1^{-1}) : v_1 \in [w_1] \, , \
v_2 \in [w_2] \right\} \, .
\end{equation}
\index{$d_{\mathcal T} ([w_1] , [w_2])$} Note that the infimum is
equal to the following
\begin{align}\label{eq2.19}
{\rm inf}&\biggl\{ \frac12 \log K(g) :\, g:Y_1\to Y_2 \;\mbox{such that}\, w_2^{-1}\circ g \circ w_1 :X \to X \; \mbox{is isotopic } \nonumber\\
&\mbox{to the identity
fixing the boundary continua pointwise} \biggl\} \, .
\end{align}
For a quasi-conformal homeomorphism $w$ with $[w] \in {\mathcal T}
(X)$ the space ${\mathcal T} (w(X))$ is canonically isometric to
${\mathcal T} (X)$. We choose a reference Riemann surface $X$ with
$m$ punctures and $\ell$ boundary continua and write ${\mathcal T}
(g,m,\ell)$.
\index{$\mathcal{T}(g,m,\ell)$}

We also need the Fuchsian model \index{Fuchsian model} of the
Teichm\"uller space. (For more details see  \cite{A1}, VI B.)
Consider Riemann surfaces $X$ and $Y$ of first kind whose universal coverings
equal ${\mathbb C}_+$.  Represent $X$ and $Y$ as quotients of the
upper half-plane ${\mathbb C}_+$ \index{$\mathbb{C}_+$} by the
action of Fuchsian groups $\Gamma$ and $\Gamma_1$ of first kind, i.e. $X \cong
{\mathbb C}_+  \diagup \Gamma$, $Y \cong {\mathbb C}_+  \diagup
\Gamma_1$. Then each homeomorphism $w$ from $X$ to $Y$ lifts to a
self-homeomorphism $\tilde w$ of the upper half-plane such that
\begin{equation}\label{eq2.20}
\mbox{for each $\gamma \in \Gamma$ there is $\gamma_1^{\tilde w} \in
\Gamma_1$ such that $\tilde w \circ\gamma =
\gamma_1^{\tilde w}$} \circ\tilde w \, ,
\end{equation}
and each self-homeomorphisms of ${\mathbb C}_+$ with this property
projects to a homeomorphisms from $X$ onto $Y$. Indeed, a mapping $\tilde{w}_1$ also lifts $w$ if and only it differs from $\tilde w$ by precomposition with a covering transformation for the covering $\mathbb{C}_+\to Y$.


Beltrami differentials and quadratic differentials on $X$ lift to
Beltrami differentials $\tilde\mu$ and quadratic differentials
$\tilde\phi$ on ${\mathbb C}_+$ with the invariance property
\begin{equation}\label{eq2.21}
\tilde\mu \circ \gamma = \tilde\mu \,
\frac{\gamma'}{\overline{\gamma'}} \quad \mbox{for} \ \gamma \in
\Gamma \qquad \mbox{(automorphic $(-1,1)$-forms)} \, ,
\end{equation}
\begin{equation}\label{eq2.22}
\tilde\phi \circ \gamma \cdot \gamma'^2 = \tilde\phi  \quad
\mbox{for} \ \gamma \in \Gamma \qquad \mbox{(automorphic $2$-forms)}
\, .
\end{equation}
A self-homeomorphism of the closed upper half-plane $\bar{\mathbb
C}_+$ (a self-homeomor\-phism of the Riemann sphere ${\mathbb P}^1$, respectively) is called normalized if it maps $0$ to $0$,
$1$ to $1$ and $\infty$ to $\infty$.
Note that each quasiconformal
self-homeomorphism of ${\mathbb C}_+$ extends to a
self-homeomorphism of $\bar{\mathbb C}_+$ (\cite{A2}).

Let $\Gamma$ be a Fuchsian group. Denote by $Q_{\rm norm} (\Gamma)$
\index{$Q_{\rm norm} (\Gamma)$} the set of normalized quasiconformal
self-homeomorphisms of ${\mathbb C}_+$ that satisfy \eqref{eq2.20}
for another Fuchsian group $\Gamma_1$. The Beltrami
differentials on ${\mathbb C}_+$ that satisfy \eqref{eq2.21} are in
one-to-one correspondence to elements of $Q_{\rm norm} (\Gamma)$.
Indeed, associate to each Beltrami differential $\mu$ on ${\mathbb
C}_+$ the Beltrami differential $\hat\mu$ on ${\mathbb C}$, for
which $\hat\mu (z) = \mu (z)$, $z \in {\mathbb C}_+$, $\hat\mu (z) =
\bar\mu (\bar z)$, $z \in {\mathbb C}_-$ (${\mathbb C}_-$ denotes
the lower half-plane). There is a unique normalized solution
\index{normalized solution} $w$ of the equation $w_{\bar z} =
\hat\mu (z) \, w_z$ on the complex plane. It maps ${\mathbb C}_+$
onto itself. Its restriction to ${\mathbb C}_+$ is denoted by
$w^{\mu}$. \index{$w^{\mu}$} \index{$W^{\mu}$} Let $\Gamma_1 = \Gamma^{\mu}$
\index{$\Gamma^{\mu}$} be the group $w^{\mu} \circ \gamma \circ
(w^{\mu})^{-1}$, $\gamma \in \Gamma$. This is a Fuchsian group. If
$\mu$ satisfies \eqref{eq2.21} then $w^{\mu}$ satisfies
\eqref{eq2.20} for $\Gamma_1 = \Gamma^{\mu}$.
Hence, $w^{\mu}$ induces a
quasiconformal mapping $W^{\mu}:\mathbb{C}_+\diagup \Gamma \to \mathbb{C}_+ \diagup \Gamma_1$.


Two elements of $Q_{\rm norm} (\Gamma)$ are called equivalent iff
their restrictions to the real axis coincide. The Teichm\"uller
space ${\mathcal T} (\Gamma)$ \index{$\mathcal{T} (\Gamma)$} is
defined as set of equivalence classes of elements of $Q_{\rm norm}
(\Gamma)$. Two mappings $w^{\mu}
, w^{\nu} \in Q_{\rm norm} (\Gamma)$ are equivalent iff the mappings
$w^{\mu} , w^{\nu}$ of ${\mathbb C}_+$ induce Teichm\"uller
equivalent mappings $W^{\mu} , W^{\nu}$ on  $X = {\mathbb C}_+
\diagup \Gamma$ (\cite{A1}, VI B, Lemma 2).


Let $\mu$ be a Beltrami differential on $X$, let $\tilde{\mu}$ be its
lift to $\mathbb{C}_+$ and let $w^{\tilde{\mu}}$ be the normalized
solution of the Beltrami equation on $\mathbb{C}_+$  for
$\tilde{\mu}$. The projection of $w^{\tilde{\mu}}$ to $X$ is
denoted by $W^{\mu}$.
For later use we give the following definition.

\begin{defn}\label{def2.1} For each Beltrami differential
$\mu$ on $X$ the homeomorphism $W^{\mu}$ is called the
normalized solution of the Beltrami equation on $X$ for the Beltrami
differential $\mu$.
\end{defn}
\index{normalized solution} Also, we assign to $\mu$ the element
$[W^{\mu}]$ of the Teichm\"uller space ${\mathcal T} (X)$. We use
the notation
$\{\mu\}$ for $[W^{\mu}]$. \index{$\{\mu\}$}
The obtained mapping from ${\mathcal T} (\Gamma)$ to ${\mathcal T}
(X)$ is a bijection.

Let now $X$ be a Riemann surface of genus $g$ with $m$ punctures and
no boundary continuum, i.e. $X=X^c \setminus E$ for a closed Riemann surface $X^c$ with set of distinguished points $E$.
We assume that the universal covering of $X$
equals ${\mathbb C}_+$. The Teichm\"uller space is denoted by
${\mathcal T} (X) \cong {\mathcal T} (g,m,0)$. Instead of ${\mathcal
T} (g,m,0)$
\index{$\mathcal{T}(g,m)$}
we will write ${\mathcal T} (g,m)$. Teichm\"uller's theorem
implies the following.


Denote by $QC(X)$ the \index{$QC(X)$} set of quasiconformal
homeomorphisms of $X$ onto another Riemann surface. The mapping
$QC(X) \overset{[ \ ]}{\longrightarrow} {\mathcal T} (X)$ assigns to
each element $w \in QC(X)$ its class $[w]$ in the Teichm\"uller
space ${\mathcal T} (X)$. The Teichm\"uller space is equipped with
the Teichm\"uller metric $d_{\mathcal T}$. Associate to each
non-trivial class $[w] \in {\mathcal T} (X)$ a (unique up to
composition with conformal mappings) extremal quasiconformal
homeomorphism in this class. We obtain a bijection from non-trivial
elements of the Teichm\"{u}ller space to Beltrami differentials of the form
$k \, \frac{\vert \phi \vert}{\phi}$ on $X$, where $k$ is a constant
in $(0,1)$ and $\phi$ is a holomorphic quadratic differential on $X$,
that extends to a meromorphic quadratic differential on $ X^c$ with at most simple poles at the points of $E$.
Notice that a  meromorphic quadratic differential on $X^c$ is integrable iff it has at most simple poles.
Assign to the
trivial class in the Teichm\"{u}ller space (corresponding to a conformal
extremal mapping) the zero quadratic differential. We obtain a
bijection of the Teichm\"{u}ller space to the space of integrable
holomorphic quadratic differentials of norm less than $1$. The real
dimension of the space of integrable holomorphic quadratic
differentials on $X$ is equal to $6g - 6 + 2m$. It can be proved
that the Teichm\"uller space ${\mathcal T} (X)$ is homeomorphic to
the unit ball in the Banach space of such quadratic differentials.


There is a unique conformal structure on ${\mathcal T} (g,m)$ with
the following property. Take any family of complex structures in
$QC(X)$ whose Beltrami differentials depend holomorphically on
certain complex parameters. More precisely, for almost all $x\in X$ the Beltrami differentials at $x$ of the members of the family depend holomorphically on the complex parameters. Then the equivalence classes in
${\mathcal T} (X)$ depend holomorphically on the complex parameters.


The explicit construction uses the Fuchsian model (see \cite{Be3}). Write $X = {\mathbb C}_+ \diagup \Gamma$. Let $\mu$ be a Beltrami
differential on ${\mathbb C}_+$ satisfying \eqref{eq2.21}. Let
$w_{\mu}$ be the unique normalized solution of the Beltrami equation
on the Riemann sphere with Beltrami coefficient equal to $\mu$ on
${\mathbb C}_+$ and equal to $0$ on ${\mathbb C}_-$. The mapping
$w_{\mu}$ does not map ${\mathbb C}_+$ onto ${\mathbb C}_+$, but if
the Beltrami differentials depend holomorphically on complex
parameters then the mappings $w_{\mu}$ depend holomorphically on
them. The mappings $w_{\mu}$ are conformal on the lower half-plane
${\mathbb C}_-$. Moreover, for two Beltrami differentials $\nu$ and
$\mu$ the equality $w^{\mu} = w^{\nu}$ holds on ${\mathbb R}$ iff
the equality $w_{\mu} = w_{\nu}$ holds on ${\mathbb R}$ (and hence
the latter equality holds on ${\mathbb C}_-$) (\cite{A1}, VI B, Lemma 1). Consider the
Schwarzian derivative
${\mathfrak S}
(w_{\mu} \mid {\mathbb C}_-)$. \index{$\mathfrak{S} (w_{\mu} \mid
{\mathbb C}_-)$} (For a locally conformal mapping $f$ on an open set
in ${\mathbb P}^1$ the Schwarzian derivative is defined as
${\mathfrak S} (f) = \left( \frac{f''}{f'} \right)' - \frac12 \left(
\frac{f''}{f'} \right)^2$.)


The Schwarzian derivatives of M\"obius transformations equal zero
and ${\mathfrak S} (f \circ g) = {\mathfrak S} (f) \circ g \cdot
(g')^2 + {\mathfrak S} (g)$ for locally conformal  mappings $f$ and
$g$. This implies that $\phi_{\mu} = {\mathfrak S} (w_{\mu} \mid
{\mathbb C}_-)$ is a quadratic differential on ${\mathbb C}_-$ which
satisfies \eqref{eq2.22} (with respect to $\Gamma$ acting on
${\mathbb C}_-$). If $\mu$ depends holomorphically on parameters
then so does $\phi_{\mu}$. We obtained a map $\mu \to \phi_{\mu} =
{\mathfrak S} (w_{\mu} \mid {\mathbb C}_-)$ from the set of Beltrami
differentials on ${\mathbb C}_+$ satisfying \eqref{eq2.21} to the
space of holomorphic quadratic differentials on ${\mathbb C}_-$. The mapping
satisfies the condition
\begin{equation}\label{eq2.22bis}
\phi_{\mu} = \phi_{\nu} \quad {\rm if} \quad w^{\mu} , w^{\nu} \in
Q_{\rm norm}(\Gamma) \; {\rm  are \; equivalent.}
\end{equation}
\index{Schwarz ! Schwarzian derivative}


Note that for each holomorphic function $f$ on a simply connected
domain in the complex plane there is a meromorphic function $w$ in
the domain, unique up to a M\"obius transformation for which
${\mathfrak S} (w) = f$. There is an explicit way to find such a
function.


Consider the Banach space $B({\mathbb C}_- , \Gamma)$ of holomorphic
quadratic differentials in ${\mathbb C}_-$ \index{$B({\mathbb C}_- , \Gamma)$}
with norm $\Vert \varphi \Vert = {\rm sup} \,
\vert y^2 \, \varphi (z) \vert$. By the condition \eqref{eq2.22bis} the mapping
\begin{equation}\label{eq2.23}
{\mathcal T} (X) \cong {\mathcal T} (\Gamma) \ni \{ \mu \} \to
\phi_{\mu}
\end{equation}
is well defined. It can be proved that it
defines a homeomorphism of ${\mathcal T} (X)$ onto an open subset of
the unit ball of $B({\mathbb C}_- , \Gamma)$ (\cite{A1}). This homeomorphism is
called Bers embedding. The complex structure on ${\mathcal T} (X)$
induced by this homeomorphism from $B({\mathbb C}_- , \Gamma)$ is
the desired one. For Beltrami differentials $\mu$ on
$X$ depending holomorphically on parameters, the Teichm\"uller
classes $[W^{\mu}] = \{ \mu \}$ also depend holomorphically on parameters, but the homeomorphisms
$W^{\mu}$ do possibly not  have
this property.

\bigskip

\noindent {\bf Teichm\"uller discs.}
Let $X$ be \index{Teichm\"{u}ller !
disc} a Riemann surface of genus $g$ with $m$ punctures with
universal covering ${\mathbb C}_+$. Write $X = {\mathbb C}_+ \diagup
\Gamma$ for a Fuchsian group $\Gamma$. Let $\phi$ be a meromorphic
quadratic differential on $X$ which is holomorphic on $\,X^c
\,$ except, maybe, at some punctures, where it may have simple
poles. For each $z = r \, e^{i\theta} \in {\mathbb D}$ we consider
the Beltrami differential $\mu_z = z \, \frac{\vert \phi
\vert}{\phi} = r \, \frac{\vert e^{-i\theta} \, \phi
\vert}{e^{-i\theta} \, \phi}$. Notice that the absolute value $|\mu_z|$
of $\mu_z$ is the constant function $|z|$ on $\mathbb{C}_+$.
Each Beltrami differential $\mu_z$
defines a unique element $w^{\mu_z} \in Q_{\rm norm} (\Gamma)$,
equivalently a unique normalized solution $W^{\mu_z} \in QC(X)$
associated to $\mu_z$, and a unique Teichm\"uller class $[W^{\mu_z}]
= \{\mu_z\} \in {\mathcal T} (X) \cong {\mathcal T} (\Gamma)$. The
mapping
\begin{equation}\label{eq2.24}
{\mathbb D} \ni z \to \{ \mu_z \} \in {\mathcal T} (X)
\end{equation}
is a holomorphic embedding. For each $z \ne 0$ the homeomorphism
$W^{\mu_z}$ is a Teichm\"uller map. Let $z_1 , z_2 \in {\mathbb D}$,
$z_1 \ne z_2$. Up to positive multiplicative constants the quadratic differential of $W^{\mu_{z_j}}$ equals $e^{-i\theta_j}\phi$, $j=1,2,$ i.e. the quadratic differentials differ by  a constant factor. Hence, by Lemma 9.1 of  \cite{Let}
the composition $W^{\mu_{z_2}} \circ
(W^{\mu_{z_1}})^{-1}$ is a Teichm\"uller mapping of the form $W^{\mu_{z'}}$ for a point $z'\in \mathbb{D}$. By formula (9.5) of \cite{Let} the point $z'$ is real, if $z_1$ and $z_2$ are real. One can show that the absolute value
of the Beltrami differential of  $W^{\mu_{z_2}} \circ
(W^{\mu_{z_1}})^{-1}$ is a constant function on $\mathbb{C}_+$ that equals
\begin{equation}\label{eq2.25}
k = |\mu_{W^{\mu_{z_2}} \circ (W^{\mu_{z_1}})^{-1}}| = \left\vert \frac{z_2
- z_1}{1-z_2 \, \bar z_1} \right\vert \, .
\end{equation}
Hence, with $K=\frac{1+k}{1-k}$ the value
$$
\frac12 \log K = \frac{1}{2} \log \frac{|1-z_2\overline{z_1}|+|z_2-z_1|}{|1-z_2\overline{z_1}|-|z_2-z_1|}
$$
is equal to the distance of $z_1$ and $z_2$ in the Poincar\'e
metric on the unit disc. (For details see, e.g. \cite{Let}.) Thus the mapping
${\mathbb D} \ni z \to \{ \mu_z \} \in {\mathcal T} (X)$ is an
isometric proper holomorphic embedding of the disc with Poincar\'e metric
into Teichm\"uller space with Teichm\"uller metric. It is called the
Teichm\"uller disc associated to $\phi$. We will denote its image in
${\mathcal T} (X)$ by ${\mathcal D}_{X,\phi}$.
\index{$\mathcal{D}_{X,\phi}$}

\bigskip

\noindent {\bf The modular group.} Consider a closed Riemann surface of
genus $g$ with $m$ punctures, denoted by $X=X^c\setminus E$. Here $X^c$ is a closed Riemann surface with a set $E$ of $m$ distinguished points.
A quasiconformal \index{modular
group} self-homeomorphism $\varphi$ of  $X$
induces a
mapping $\varphi^*$ of the Teichm\"uller space ${\mathcal T} (X)
\cong {\mathcal T} (g,m)$ to itself. It is defined as follows. For
each homeomorphism $w \in QC(X)$, $w:X\to Y$ the composition $w \, \circ \,
\varphi : X \to Y$ is another quasiconformal homeomorphism.
Its class $[w \circ \varphi] \in
{\mathcal T} (g,m)$ depends only on the class $[w]$ of $w$. Put
$\varphi^* ([w]) = [w \circ \varphi]$.
For each quasiconformal self-homeomorphism $\varphi$ of $X$
the mapping $\varphi^*$ is an
isometry on the Teichm\"uller space ${\mathcal T} (g,m)$.
Moreover,
it maps ${\mathcal T} (g,m)$ biholomorphically onto itself. The
mapping $\varphi^*$ is called the modular transformation of
$\varphi$.

Notice that
$\varphi^*$ depends only on the mapping class of $\varphi$ in the mapping class group $\mathfrak{M}(X;\emptyset)= {\mathfrak M} (X^c , \emptyset , E)$, in other words, isotopic self-homeomorphisms of $X$
have the same modular transformation. The modular group is
isomorphic to the mapping class group ${\mathfrak M} (X) = {\mathfrak M} (X^c , \emptyset , E)$.
For two quasiconformal self-homeomorphism $\varphi_1$ and $\varphi_2$ of the punctured Riemann surface $X$
the equality $(\varphi_1\circ\varphi_2)^*=\varphi_1^* \varphi_2^*$ holds, because $[w\circ(\varphi_1\circ\varphi_2)]=\varphi_2^*([w\circ\varphi_1])=
\varphi_1^*\varphi_2^*([w])$.
Hence, the set of modular transformations forms a group, called
the modular group. It is often denoted by ${\rm Mod} (g,m)$.
The quotient ${\mathcal T}
(g,m) \diagup {\rm Mod}(g,m)$ can be identified with the Riemann
space of conformally equivalent complex structures, i.e. with the
moduli space of Riemann surfaces of genus $g$ with $m$ punctures.
\index{${\rm Mod}(g,m)$}
\index{moduli space}

\bigskip

\noindent {\bf Royden's theorem.} The following deep theorem of
Royden \cite{Roy} has many applications.
\index{Theorem ! Royden}
\begin{thm}[Royden]\label{thm2.4} The Teichm\"uller metric on the
Teichm\"uller space ${\mathcal T}(g,m)$ is equal to the Kobayashi
metric. \index{metric ! Kobayashi} In other words, a holomorphic
mapping from the unit disc ${\mathbb D}$ with Poincar\'e metric
$\frac{\vert dz \vert}{1-\vert z \vert^2}$ into ${\mathcal T} (g,m)$
with Teichm\"uller metric is a contraction. (Equivalently, a
holomorphic map from ${\mathbb C}_+$ with hyperbolic metric
$\frac{\vert dz \vert}{2 \, y}$ into ${\mathcal T}(g,m)$ is a
contraction.) \index{metric ! hyperbolic}
\end{thm}

\bigskip

\noindent {\bf Configuration space and Teichm\"uller space.}
We will relate geometric $n$-braids
to paths in the Teichm\"uller space ${\mathcal T} (0,n+1)$ of the
Riemann sphere with $n+1$ punctures.
The reference Riemann surface
will be denoted by $X_0 = {\mathbb C} \backslash E_n^0 = {\mathbb P}^1 \backslash
(\{ \infty \} \cup E_n^0)$, where $E_n^0 = \left\{ 0 , \frac1n ,
\ldots \frac{n-1}n \right\}$. Let $E_n^1 \subset {\mathbb C}$ be
another set containing exactly $n$ points. For a homeomorphism from
${\mathbb C} \backslash E_n^0$ onto ${\mathbb C} \backslash E_n^1$
we will use the same notation for this homeomorphism and for the
extension of this homeomorphism to a self-homeomorphism of ${\mathbb
P}^1$ which maps $E_n^0$ to $E_n^1$ and $\infty$ to $\infty$.

Consider the set  $QC_{\infty} (0,n+1)$ \index{$QC_{\infty}(0,n+1)$} of
orientation preserving quasiconformal homeomorphisms $w:\mathbb{C}\setminus E_n^0 \to
\mathbb{C}\setminus E_n^1$ whose extensions to $\mathbb{P}^1$ fix $\infty$. We equip this set with the topology for which a neighbourhood basis
of an element $w_0$ is given by the sets
\begin{align*}
N_{\varepsilon}(w_0)=&\big\{ w\in QC_{\infty} (0,n+1):\\
& \|\mu_{w}-\mu_{w_0}\|_{\infty}<\varepsilon,\,
|w(0)-w_0(0)|<\varepsilon,\, |w(\frac{1}{n})-w_0(\frac{1}{n}  )|<\varepsilon\big\}.
\end{align*}
Convergence in this topology implies uniform convergence on compact subsets of $\mathbb{C}$ (of the extensions of the mappings across the punctures). Indeed, let $\mathfrak{a}_{w,w_0}:\mathbb{C}\,\toitself$ be the complex affine
mapping that takes $w(0)$ to $w_0(0)$ and $w(\frac{1}{n})$ to $w_0(\frac{1}{n})$. Then $\mathfrak{a}_{w,w_0}\circ w$ takes $0$ to $w_0(0)$ and $\frac{1}{n}$ to $w_0(\frac{1}{n})$, and  for any sequence of elements $w_n\in N_{\varepsilon}(w_0)$ with $\varepsilon_n\to 0$ the sequence
$\mathfrak{a}_{w_n,w_0}$ converges to the identity uniformly on compacts.
By formulas \eqref{eq2.14'}  and \eqref{eq2.14''}
the $L^{\infty}$-norm of the Beltrami coefficient of the mapping  $\mathfrak{a}_{w_n,w_0}\circ w_n \circ w_0^{-1}$ converges to $0$ for
any sequence $w_n\in N_{\varepsilon}(w_0),\,\varepsilon_n\to 0 $.  The mappings $\mathfrak{a}_{w_n,w_0}\circ w_n\circ w_0^{-1}$ fix the points $w_0(0)$, $w_0(\frac{1}{n})$, and $\infty$, i.e. they are normalized solutions of Beltrami equations on $\mathbb{C}$ with Beltrami coefficients converging to $0$ in the $L_{\infty}$-norm. Hence, these mappings converge to the identity uniformly on compacts (see \cite{A1}, Ch. V, B, Lemma 1 and the proof of Theorem 3).

Each self-homeomorphism $\psi$ of ${\mathbb C}$ acts diagonally on
the configuration space $C_n ({\mathbb C})$. Denote this action
again by $\psi$:
\begin{equation}\label{eq2.28}
\psi: (z_1 , \ldots , z_n) \to (\psi (z_1) , \ldots , \psi (z_n))
,\;\; (z_1 , \ldots , z_n) \in C_n ({\mathbb C}) \, .
\end{equation}
The action descends to an action on the symmetrized configuration space $C_n
({\mathbb C}) \diagup {\mathcal S}_n$.


Let ${\mathcal A}$ be the set of complex affine mappings on the \index{$\mathcal{A}$}
complex plane (equivalently, the set of M\"obius transformations on the Riemann sphere that fix $\infty$). Each element ${\mathfrak a} \in {\mathcal A}$
\index{$\mathfrak{a}$} has the form ${\mathfrak a}(z) = az+b$, $z
\in {\mathbb C}$. Here $b \in {\mathbb C}$ and $a \in {\mathbb C}^*
= {\mathbb C} \backslash \{ 0 \}$ are constants. ${\mathcal A}$ has
the complex structure of ${\mathbb C}^* \times {\mathbb C}$. It
forms a group under composition.

For $z=(z_1,z_2,\ldots,z_n)$ we put $\mathfrak{a}_z(\zeta)= \frac{1}{n}\frac{\zeta-z_1}{z_2-z_1}$.  Then
\begin{align*}
\mathfrak{a}_z(\zeta)( (z_1,z_2,\ldots,z_n))= &\big(0,\, \frac{1}{n},\,\frac{1}{n} \, \frac{z_3 - z_1}{z_2 - z_1},\,\ldots,\, \frac{1}{n} \, \frac{z_n - z_1}{z_2 -
z_1}\big)\\
\in& \{0\}\times \{\frac{1}{n}\} \times
C_{n-2} ({\mathbb C} \backslash \{0,\frac{1}{n}\})\,.
\end{align*}
Recall that
$$
C_{n-2} \big({\mathbb C} \backslash \{0,\frac{1}{n}\}\big) =\big \{ (z_3 ,
\ldots , z_n) \in \big({\mathbb C} \, \backslash
\{0,\frac{1}{n}\}\big)^{n-2} :\mbox{ $ z_i \ne z_j\, $ for $\, i\ne
j$}\big\}\,.
$$
The mapping
\begin{align*}
(z_1,z_2,\ldots,z_n) \to \big(z_1,\, z_2-z_1,\,\frac{1}{n} \, \frac{z_3 - z_1}{z_2 - z_1},\,\ldots,\, \frac{1}{n} \, \frac{z_n - z_1}{z_2 -
z_1}\big)
\end{align*}
is a holomorphic isomorphism from $C_n(\mathbb{C})$ onto $\mathbb{C}\times \mathbb{C}^*\times
C_{n-2} ({\mathbb C} \backslash \{0,\frac{1}{n}\})$.
We denote by $\mathcal{P}_{\mathcal{A}}$ the projection from $C_n(\mathbb{C})$ onto the quotient $C_n(\mathbb{C})\diagup {\mathcal{A}}$ by the action of $\mathcal{A}$. The quotient $C_n(\mathbb{C})\diagup {\mathcal{A}}$ with the inherited complex structure is isomorphic to $\{0\}\times\{\frac{1}{n}\}\times  C_{n-2} ({\mathbb C} \backslash \{0,\frac{1}{n}\})$.
Denote by ${\rm Is_n}$ the isomorphism ${\rm Is_n}: C_n(\mathbb{C})\diagup {\mathcal{A}} \to \{0\}\times\{\frac{1}{n}\}\times  C_{n-2} ({\mathbb C} \backslash \{0,\frac{1}{n}\})$ that assigns to each element of the quotient  $C_n(\mathbb{C})\diagup {\mathcal{A}}$ the normalized element of $C_n(\mathbb{C})$ whose first two coordinates are $0$ and $\frac{1}{n}$.
The composition $\mathcal{P}_{\mathcal{A},{\rm n}}\stackrel{def}={\rm Is_n}\circ \mathcal{P}_{\mathcal{A}}$ is equal to the mapping
$(z_1,z_2, \ldots, z_n) \xrightarrow {{\rm Is_n}\circ \mathcal{P}_{\mathcal{A}}} \big(0,\, \frac{1}{n},\,\frac{1}{n} \, \frac{z_3 - z_1}{z_2 - z_1},\,\ldots,\, \frac{1}{n} \, \frac{z_n - z_1}{z_2 -
z_1}\big) $.
\index{$\mathcal{P}_{\mathcal A}$}

Recall that for each self-homeomorphism $w$ of ${\mathbb P}^1$
the map $e_n (w)$ is defined by $e_n (w) = \Bigl(w (0) ,
w\left(\frac1n\right) , \ldots ,$ $w \left( \frac{n-1}n
\right)\Bigl)$.
Recall also that two complex
structures $w_1 , w_2$ on ${\mathbb C} \backslash E_n^0$ are
Teichm\"uller equivalent if there is a conformal mapping $c : w_1
({\mathbb C} \backslash E_n^0) \to w_2 ({\mathbb C} \backslash
E_n^0)$ such that $w_2^{-1} \circ c \circ w_1 : {\mathbb C}
\backslash E_n^0 \to {\mathbb C} \backslash E_n^0$ is isotopic to
the identity on ${\mathbb C} \backslash E_n^0$. Denote by the same
letters $c,w_1,w_2$ the extensions of the previous mappings to
${\mathbb P}^1$. The mapping $c$ is a M\"obius transformation that
fixes $\infty$, hence $c \in {\mathcal A}$.
This implies that for two Teichm\"uller equivalent conformal structures $w_1$ and $w_2$ on ${\mathbb C} \backslash E_n^0$ there exists a mapping $c \in \mathcal{A}$ such that
$e_n (w_2) = c(e_n (w_1))$.
Indeed, if $w_2^{-1} \circ c\circ w_1$ is isotopic to the identity on ${\mathbb C} \backslash E_n^0$ its extension to $\mathbb{C}$ fixes each point in $E_n^0$, hence $w_2$ and $c\circ w_2$ take the same values at each point of $E_n^0$.

The arguments imply that
to each Teichm\"uller class $[w] \in
{\mathcal T} (0,n+1)$ corresponds a unique class in $C_n ({\mathbb
C}) \diagup {\mathcal A}$. It is obtained as follows. Take an element $\tau\in {\mathcal T}(0,n+1)$. Represent $\tau$ by a self-homeomorphism $w\in \tau$. The element $e_n(w)\in C_n(\mathbb{C})$ is defined modulo a diagonal action of the group $\mathcal{A}$ of complex affine self-mappings of $\mathbb{C}$. We denote by $\mathcal{P}_\mathcal{T}(\tau)$ the unique element of the class in $C_n ({\mathbb
C}) \diagup {\mathcal A}$ corresponding to $\tau$, that has the form
$(0,\frac{1}{n},z_3,\ldots,z_n)$. \index{$\mathcal{P}_\mathcal{T}$}

In the following simple but useful lemma we again identify
homomorphisms between punctured surfaces and their extensions to
${\mathbb P}^1$.

\begin{lemm}\label{lemm2.6} Let $E_n^1$ and $E_n^2$ be subsets of
${\mathbb C}$, each containing exactly $n$ points. Let $w_1 :
{\mathbb C} \backslash E_n^0 \to {\mathbb C} \backslash E_n^1$ and
$w_2 : {\mathbb C} \backslash E_n^0 \to {\mathbb C} \backslash
E_n^2$ be Teichm\"uller equivalent homeomorphisms. If $w_1$ and $w_2$ take equal values on two different points $\frac{j_1}{n}$ and  $\frac{j_2}{n}$  of $E_n^0$, i.e. $w_1 \left(\frac{j_1}{n}\right) = w_2
\left(\frac{j_1}{n}\right)$ and $w_1 \left( \frac{j_2}{n} \right) = w_2 \left( \frac{j_2}{n} \right)$,
then $E_n^1 = E_n^2$ and $w_1 , w_2$ are isotopic as mappings from
${\mathbb C} \backslash E_n^0$ onto ${\mathbb C} \backslash E_n^1$ (i.e.
$w_1 \circ w_2^{-1}:\mathbb{C}\setminus E_n^2\toitself$ is isotopic to the identity.)
In particular, $e_n \circ w_1 = e_n \circ w_2$.
\end{lemm}
\begin{figure}[H]
\begin{center}
{ $ \xymatrix
   { & X^0 \ar[dl]_{w_1} \ar[dr]^{w_2} & \\
    \mathbb{C}\setminus E_n^1  \ar[rr]_{c} & & \mathbb{C}\setminus E_n^2
   }
$}
\end{center}
\caption{Teichm\"uller equivalent mappings}
\end{figure}

\noindent {\bf Proof.} Under the conditions of the lemma there is an affine map $c \in {\mathcal A}$ such that the
mapping $w_2^{-1} \circ c \circ w_1$ is isotopic to the identity
through self-homeomorphisms of ${\mathbb C} \diagup E_n^0$. Then $w_2^{-1} \circ c \circ w_1$
fixes $E_n^0$ pointwise, i.e. $c \circ w_1 \mid E_n^0 = w_2 \mid
E_n^0$. Since $w_1 \left(\frac{j_1}{n}\right) = w_2
\left(\frac{j_1}{n}\right)$ and $w_1 \left( \frac{j_2}{n} \right) = w_2 \left( \frac{j_2}{n} \right)$,
$c$ fixes two points in ${\mathbb C}$, hence
$c$ is the identity. Thus $E_n^1 = w_1 (E_n^0) = w_2 (E_n^0) =
E_n^2$ and $w_2^{-1} \circ w_1$ is isotopic to the identity through
self-homeomorphisms of ${\mathbb C} \backslash E_n^0$. In other
words, there is a continuous family $\varphi^t$, $t \in [0,1]$, of
self-homeomorphisms of ${\mathbb C} \mid E_n^0$ with $\varphi^0 =
w_2^{-1} \circ w_1$ and $\varphi^1 = {\rm id}$. The required isotopy
is $w_2 \circ \varphi^t$, $t \in [0,1]$. The equality $e_n \circ w_1
= e_n \circ w_2$ follows. \hfill $\Box$

\medskip

The following theorem was proved by Kaliman  \cite{Ka} and later independently by Bers and Royden \cite{BeRo}. Our proof follows the approach of Kaliman.
\index{Kaliman ! Theorem}

\begin{thm}\label{thm2.1a}
The mapping $\mathcal{P}_{\mathcal{T}}:\mathcal{T}(0,n+1)\to  \{0\}\times\{\frac{1}{n}\}\times      C_{n-2} ({\mathbb C} \backslash \{0,\frac{1}{n}\})$ is a holomorphic covering. Hence,
the Teichm\"uller space $\mathcal{T}(0,n+1)$ is the holomorphic universal
covering space of the space $\{0\}\times\{\frac{1}{n}\}\times   C_{n-2} ({\mathbb C} \backslash \{0,\frac{1}{n}\})$.
\end{thm}

\noindent {\bf Proof.}
The following lemma implies the theorem. \hfill $\Box$

\medskip

In the Lemma \ref{lem2.1a} we consider $\{0\}\times\{\frac{1}{n}\}\times      C_{n-2} ({\mathbb C} \backslash \{0,\frac{1}{n}\})$ as subspace of $C_n(\mathbb{C})$ and denote points of this space by $\tilde {E}_n$.

\begin{lemm}\label{lem2.1a}
For any point $\tilde{E}^*_n\in
\{0\}\times\{\frac{1}{n}\}\times      C_{n-2} ({\mathbb C} \backslash \{0,\frac{1}{n}\})$ there exists a neighbourhood $U$ in $\{0\}\times\{\frac{1}{n}\}\times      C_{n-2} ({\mathbb C} \backslash \{0,\frac{1}{n}\})$, such that for
any preimage $\tau_0\in\mathcal{P}_{\mathcal{T}}^{-1}(\tilde{E}^*_n)$ there exists a holomorphic mapping $g_U: U\to \mathcal{T}(0,n+1)$ that takes the value $\tau_0$ at the point $\tilde{E}^*_n$ and is a local inverse of $\mathcal{P}_{\mathcal{T}}$.

Moreover, for any given homeomorphism $w_{\tilde{E}_n^*}^0\in  QC_{\infty}(0,n+1)$,  such that
$e_n(w_{\tilde{E}_n^*}^0)=\tilde{E}_n^*$ and $[w_{\tilde{E}_n^*}^0]=\tau_0$, there exists a smooth mapping\\
$\quad \quad U\ni\tilde{E}_n \xrightarrow{g_U^1} w_{\tilde{E}_n}\in QC_{\infty}(0,n+1)$
with $w_{\tilde{E}_n^*}=w_{\tilde{E}_n^*}^0$ and $e_n(w_{\tilde{E}_n})= \tilde{E}_n$,\\ such that
$g_U(\tilde{E}_n)= [w_{\tilde{E}_n}]= [g_U^1(\tilde{E}_n)]$, and hence,   $\mathcal{P}_{\mathcal{T}}([w_{ \tilde{E}_n}])= \mathcal{P}_{\mathcal{T}}(g_U(\tilde{E}_n) )= \tilde{E}_n$.
\end{lemm}

\noindent{\bf Proof.} We consider the
family $v^{\mathbb{D}}_z:\overline{\mathbb{D}}\toitself$ of self-homeomorphisms of the unit disc, \index{$v^{\mathbb{D}}_z$}
\begin{equation}\label{eq2.29}
v^{\mathbb{D}}_z(\zeta)=\frac{\rho(\zeta) z+\zeta}{1+\rho(\zeta) z\overline{\zeta}}\,, \zeta \in \mathbb{D},
\end{equation}
where the parameter $z$ runs over a disc $\varepsilon\mathbb{D}$ with center $0$ and small radius $\varepsilon$. Here $\zeta\to\rho(\zeta)$ is a (real) non-negative smooth function on $\mathbb{D}$, that vanishes near the unit circle, equals $1$ in a neighbourhood of the closed disc $\overline{\varepsilon\mathbb{D}}$, and satisfies the inequality $|\rho(\zeta)|\leq 1$ on $\mathbb{D}$.
For each $z \in  \varepsilon\mathbb{D}$  the restriction $v^{\mathbb{D}}_z|  \overline{\varepsilon\mathbb{D}}$ equals $\zeta\to \frac{z+\zeta}{1+z\overline{\zeta}}$.
This mapping takes $0$ to $z$. The mapping $v^{\mathbb{D}}_z$ is equal to the identity on the unit circle.
For each $\zeta\in \overline{{\mathbb{D}}}$ the mapping $z\to v_z^{\mathbb{D}}(\zeta),\,z\in \mathbb{D},$ is holomorphic.
The following equalities hold.
\begin{align}\label{eq2.29'}
v_z^{\mathbb{D}}(\zeta)-\zeta=& \frac{\rho(\zeta) z (1-|\zeta|^2)}{1+z\rho(\zeta)\overline{\zeta}  }\,,\nonumber\\
\partial_{\overline{\zeta}}(v_z^{\mathbb{D}}(\zeta)-\zeta)=
&\frac{z \partial_{\overline{\zeta}}\rho(\zeta)(1-|\zeta|^2)-\rho(\zeta)z {\zeta}}{1+z\rho(\zeta)\overline{\zeta}}-
\frac{\rho(\zeta) z (1-|\zeta|^2)(z \partial_{\overline{\zeta}}\rho (\zeta)\overline{\zeta}+z\rho(\zeta))}
{ (1+z\rho(\zeta)\overline{\zeta})^2   }\,,\quad\nonumber\\
\partial_{\zeta}(v_z^{\mathbb{D}}(\zeta)-\zeta) = &\frac{z\partial_{\zeta}\rho(\zeta)(1-|\zeta|^2)-\rho(\zeta)z\overline{\zeta}} {1+z\rho(\zeta)\overline{\zeta}}-\frac{\rho(\zeta) z (1-|\zeta|)^2 \, z(\partial_{\zeta}
\rho)(\zeta)\overline{\zeta} }{(1+z\rho(\zeta)\overline{\zeta})^2      }\,.\quad\quad
\end{align}

Put $C\stackrel{def}=\max_{|\zeta|\leq 1} |(\partial_{\zeta}\rho) (\zeta)|= \max_{|\zeta|\leq 1} |(\partial_{\overline{\zeta}}\rho) (\zeta)|  $.
For each $z\in \varepsilon \mathbb{D}$ we extend $v^{\mathbb{D}}_z$ to $\mathbb{P}^1$ by the identity outside the unit disc. Denote the extended function again by $v^{\mathbb{D}}_z$.
By equation \eqref{eq2.29'} for each $z\in\varepsilon \mathbb{D}$ the mapping  $v^{\mathbb{D}}_z(\zeta)-\zeta,\, \zeta \in \mathbb{P}^1,$ is  smooth and vanishes outside the unit disc.  Its supremum norm and the supremum norm of its differential do not exceed $C' \varepsilon$ for a constant $C'$ depending on $C$. These facts imply, that if $\varepsilon$ is small, for each $z \in\varepsilon\mathbb{D}$ the mapping  $v^{\mathbb{D}}_z:\mathbb{P}^1\toitself$ is a diffeomorphism. Indeed, $v^{\mathbb{D}}_z$ is a local diffeomorphism and it is  $C'\varepsilon$-close to the identity in the supremum norm, hence it is
injective, if $\varepsilon$ is small. Hence, $v^{\mathbb{D}}_z$ is a diffeomorphism from $\mathbb{P}^1$ onto its image.   Since it is the identity outside the unit disc, it is a diffeomorphism from $\mathbb{P}^1$ onto itself, and it maps the unit disc onto itself. Moreover,
the Beltrami coefficient
$\mu_{v_z^{\mathbb{D}}}(\zeta)= \frac{\partial_{\overline{\zeta}}v_z^{\mathbb{D}}(\zeta) }{\partial_{\zeta}v_z^{\mathbb{D}}(\zeta)}$  of $v^{\mathbb{D}}_z$ for $|\zeta|<1$  satisfies the inequality $\sup_{\zeta\in \mathbb{D}}|\mu_z^{\mathbb{D}}(\zeta)|\leq C'' \varepsilon$, and, hence, is small for each $z\in \varepsilon\mathbb{D}$ if $\varepsilon$ small.
Moreover, the mapping $z\to \mu_{v_z^{\mathbb{D}}}(\zeta), \, z\in\varepsilon\mathbb{D},$
is holomorphic for each $\zeta \in \mathbb{D}$.

Write $\tilde{E}_n^*= (0,\frac{1}{n},z_3^0,\ldots,z_n^0) $, and
let $D_j$ be disjoint discs in $\mathbb{C}\setminus \{0,\frac{1}{n}\}$ with center $z_j^0$, $j=3,\ldots,n$. The set $U'\stackrel{def}=\{0\}\times\{1\}\times D_3\times\ldots\times D_n \subset \{0\}\times\{\frac{1}{n}\}\times      C_{n-2} ({\mathbb C} \backslash \{0,\frac{1}{n}\})$ is a neighbourhood of $(0,\frac{1}{n},z_3^0,\ldots,z_n^0)$ in $\{0\}\times\{\frac{1}{n}\}\times      C_{n-2} ({\mathbb C} \backslash \{0,\frac{1}{n}\})$.
Let $D_j^{\varepsilon},\, j=3,\ldots,n,$ be the disc with the same center as $D_j$ and radius equal to the radius of $D_j$ multiplied by the small number
$\varepsilon$ chosen above. We consider the neighbourhood $U \stackrel{def}=\{0\}\times\{1\}\times D_3^{\varepsilon}\times\ldots\times D_n^{\varepsilon}$, $U\subset U'$, of $(0,\frac{1}{n},z_3^0,\ldots,z_n^0)$.
Take any $\tau_0\in \mathcal{T}(0,n+1)$ for which $\mathcal{P}_{\mathcal{T}}(\tau_0)=(0,\frac{1}{n},z_3^0,\ldots,z_n^0)=\tilde{E}_n^*$ and any $w_{\tilde{E}_n^*}^0$ with $[w_{\tilde{E}_n^*}^0]=\tau_0$.
We will define a continuous family $w_{\tilde{E}_n}\in QC_{\infty}(0,n+1), \tilde{E}_n\in U,$ of
quasiconformal homeomorphisms of $\mathbb{P}^1$ with the following properties. $e_n(w_{\tilde{E}_n})=\tilde{E}_n$, $g_U(\tilde{E}_n)= [w_{\tilde{E}_n}]$, and
$\mathcal{P}_{\mathcal{T}}([w_{\tilde{E}_n}])=\tilde{E}_n$. Moreover, $w_{\tilde{E}_n^*}=w_{\tilde{E}_n^*}^0$,
and for all $\zeta \in\mathbb{P}^1 $  the Beltrami coefficient $\mu_{w_{\tilde{E}_n}}(\zeta)$ depends holomorphically on $ \tilde{E}_n\in U$.

For $j=3,\ldots,n$ we denote by $\mathfrak{a_j}\in \mathcal{A}$ a complex affine mapping that maps the unit disc onto $D_j$ and maps $0$ to $z_j^0$. Let $\tilde{E}_n=(0,\frac{1}{n},z_3,\ldots,z_n)\in U$.
Put $v_{\tilde{E}_n}(\zeta)$ equal to $\mathfrak{a}_j\circ v_{\mathfrak{a}_j^{-1}(   z_j)}^{\mathbb{D}}\circ \mathfrak{a}_j^{-1}(\zeta)$ if $\zeta\in D_j$, $j=3,\ldots,n$, and equal to $\zeta$ on the rest of $\mathbb{P}^1$. Since $v_z^{\mathbb{D}}$ is the identity near $\partial \mathbb{D}$, for each $\tilde{E}_n=(1,\frac{1}{n},z_3,\ldots,z_n)\in U$ the mapping $(\zeta,\tilde{E}_n) \to v_ {\tilde{E}_n}(\zeta)$ is smooth on $\mathbb{P}^1$,
takes $z_j^0$ to $z_j$, $j=3,\ldots,n$, and fixes $0$,$\frac{1}{n}$, and $\infty$.
$v_{\tilde{E}_n^*}$ is equal to the identity.
Moreover, for each $\zeta\in \mathbb{P}^1$ each $v_{\tilde{E}_n}$ in the family is differentiable at $\zeta$. By equations \eqref{eq2.14'} and \eqref{eq2.14''} (see also
\cite{A1}, 1D, equations (7) and (8)) for the Beltrami differential of compositions, the absolute value of its Beltrami coefficient is bounded by
$C''\varepsilon$. Put $g_U^1(\tilde{E_n}) = w_{\tilde{E}_n}= v_{\tilde{E}_n}\circ w_{\tilde{E}_n^*}^0$. Then $w_{\tilde{E}_n^*}=w_{\tilde{E}_n^*}^0$. The mapping $\tilde{E}_n\to \mu_{w_{\tilde{E}_n}} $ is uniformly continuous in $\tilde{E}_n$ and  $\zeta \in\mathbb{P}^1$ (recall that each $\mu_{w_{\tilde{E}_n}}$ vanishes outside $\mathbb{D}$). Hence, the mapping
$\tilde{E}_n\to v_{\tilde{E}_n}\circ w_{\tilde{E}_n^*}^0\in QC_{\infty}(0,n+1)$ is continuous. Moreover, the mapping $\tilde{E}_n\to \mu_{w_{\tilde{E}_n}}(\zeta)$ is holomorphic for every $\zeta \in \mathbb{P}^1$. By the choice of the complex structure on the Teichm\"uller space $\mathcal{T}(0,n+1)$ the mapping $\tilde{E}_n\to [w_{\tilde{E}_n}]\in \mathcal{T}(0,n+1)$ is holomorphic.
It is clear that $\mathcal{P}_{\mathcal{T}}([g_U^1(\tilde{E_n})])=     \mathcal{P}_{\mathcal{T}}([w_{\tilde{E}_n}])=
\tilde{E}_n=(0,\frac{1}{n},z_3,\ldots,z_n)$. The lemma is proved. \hfill $\Box$

\medskip

For later use we state the following
tautological lemma.
\begin{lemm}\label{lemm2.5} The following diagram is commutative.
$$
\xymatrix{
QC_{\infty} (0,n+1)\ar[rrr]^{[ \ ]}
 \ar[d]^{e_n}   && &\!\!\;\;{\mathcal T} (0,n+1) \ar[d]^{{\mathcal P}_{\mathcal T}} \\
C_n({\mathbb C})
\ar[r]^{{\mathcal P}_{\mathcal A}\;\;\;\;}&\!\!\;\; \; C_n({\mathbb C})\diagup {\mathcal A} \;\;\; \ar[rr]^{\;\;\;\;\;\;\;\;{\rm Is_n}\;\;\;\;}&& \{0\}  \times \{\frac{1}{n}\}&\!\!\!\!\!\!\!\!\!\!\!\!\!\!\!\!\!\!\!\!\;\times C_{n-2} ({\mathbb
C}\setminus \{0,\frac{1}{n}\})    }
$$
All mappings in the diagram are continuous.
\end{lemm}

\noindent {\bf The modular transformation corresponding to braids.}
For each $n$-braid $b\in \mathcal{B}_n$ we consider the mapping class
\begin{align*}
\mathfrak{m}_{b,\infty}=\mathcal{H}_{\infty}(\mathfrak{m}_{b})\in \mathfrak{M}(\mathbb{P}^1; \{\infty\}, E_n^0) \subset  \mathfrak{M}(\mathbb{P}^1\setminus (E_n^0\cup\{\infty\}))\,.
\end{align*}
We associate to  $\mathfrak{m}_{b,\infty}$ the modular transformation $\varphi_{b,\infty}^*$
on the Teichm\"uller space $\mathcal{T}(n+1,0)$, that is induced by a homeomorphism $\varphi_{b,\infty}$ that represents the mapping class $\mathfrak{m}_{b,\infty} $.
The modular transformation does not depend on the choice of the representing homeomorphism $\varphi_{b,\infty}$. It is convenient to choose for $\varphi_{b,\infty}$  the mapping $\varphi_1$ of a parameterizing isotopy for a geometric braid representing $b$. We denote
the mapping $\varphi_{b,\infty}^*$ on $\mathcal{T}(n+1,0)$ by $T_b$ and call it the modular transformation $T_b$
 of $b$. The mapping $T_b$ depends only on $\mathfrak{m}_{b,\infty}$, i.e. on $b\diagup \mathcal{Z}_n$.
Since braids are composed as elements of a fundamental group and modular transformations are composed as mappings, the relation
\begin{align}\label{eq2.100}
T_{b_1b_2}=T_{b_2} T_{b_1}\,
\end{align}
holds for two braids $b_1$ and $b_2$.

Vice versa, take a modular transformation $T$ on the Teichm\"uller space  $\mathcal{T}(n+1,0)$,
i.e. a mapping $T:\mathcal{T}(n+1,0)\toitself$ such that for a representative $\varphi$ of a class in $\mathfrak{M}(\mathbb{P}^1; \infty,E_n^0)$ for each element $[w]\in \mathcal{T}(n+1,0)$ the equality  $T([w])=[w\circ\varphi]$ holds.
We associate to $T$ an element
 $b\diagup \mathcal{Z}_n$ as follows.

Let $\tilde{q}_0 ={\rm Id}\in \mathcal{T}(n+1,0)$. Then
$T([{\rm Id}])=[\varphi]$. Let $\varphi_t$ be a continous family of self-homeomorphisms of $\mathbb{P}^1$ that fix $\infty$ such that $\varphi_0={\rm Id} $ and $\varphi_1=\varphi$.
Then $e_n(\varphi_0)=e_n({\rm Id})$ and $e_n(\varphi_1)=e_n(\varphi)=S_{\varphi}(e_n({\rm Id}))$ for a permutation $S_{\varphi}$ of the set $E_n^0$, since $\varphi$ maps $E_n^0$ onto itself permuting the points of $E_n^0$. Then the curve $\mathcal{P}_{\rm sym}e_n(\varphi_t),\, t\in[0,1], $ in $C_n(\mathbb{C})\diagup \mathcal{S}_n$ defines a geometric braid. Let $b$ be the braid represented by it. Then $T=T_b$. Indeed,   $\varphi_t,\, t\in[0,1]$ is a parameterizing isotopy for the geometric braid
$\mathcal{P}_{\rm sym}e_n(\varphi_t),\, t\in[0,1], $ and $\varphi_1=\varphi=\varphi_{b,\infty}$.

We proved that the mapping $\mathcal{B}_n\diagup \mathcal{Z}_n\ni b\diagup \mathcal{Z}_n \to T_b\in \mathcal{T}(n+1,0)$ is  a bijection that satisfies \eqref{eq2.100}.

\bigskip

\section{Thurston's classification of mapping
classes.}
\label{sec:2.4}
\index{Thurston classification}  Thurston's interest in
surface homeomorphisms was motivated by his celebrated geometrization
conjecture. He considered a closed surface $S$ and a
self-homeomorphism $\varphi$ of $S$. The mapping torus

$$
([0,1] \times S) \diagup (0,x) \sim (1,\varphi (x))
$$
is obtained by gluing the fiber over the point $0$ of the cylinder
$[0,1] \times S$ to the fiber over $1$ using the homeomorphism
$\varphi$. Thurston observed that for a class of homeomorphisms,
which he called pseudo-Anosov, the mapping torus admits a complete
hyperbolic metric of finite volume.  \index{pseudo-Anosov homeomorphism}
This was one of the eight
geometric structures. Moreover, Thurston gave a classification of
mapping classes of surface homeomorphisms. Here is Thurston's
theorem on classification of mapping classes.

\begin{thm}[Thurston \cite{Th}]\label{thm2.7} A self-homeomorphism of
a closed surface which is not isotopic to a periodic one is either
isotopic to a pseudo-Anosov homeomorphism or is reducible, but not
both.
\end{thm}

\index{Thurston ! Theorem}

We explain now  Thurston's notion of pseudo-Anosov homeomorphisms.
Let $S$ be a connected finite smooth oriented surface. It is either closed or homeomorphic to a surface with a finite number of punctures. We will assume from the beginning that $S$ is either closed or punctured.

A finite non-empty set of mutually disjoint Jordan curves $\{ C_1 ,
\ldots , C_{\alpha}\}$ on a connected closed or punctured oriented surface $S$ is called
admissible if no $C_i$ is homotopic
to a point in $X$, or to a puncture,
or
to a $C_j$ with $i \ne j$. Thurston calls an isotopy class $\mathfrak{m}$ of self-homeomorphisms
of $S$ (in other words, a mapping class on $S$) reducible if there is an admissible system
of curves $\{ C_1 ,\ldots , C_{\alpha}\}$ on $S$ such that
some (and, hence, each) element in $\mathfrak{m}$ maps the system to an isotopic system. In
this case we say that the system  $\{ C_1 ,
\ldots , C_{\alpha}\}$ reduces $\mathfrak{m}$. A mapping class which is not reducible is
called irreducible. 

Let $S$ be a closed or punctured surface with set $E$ of distinguished points. We say that $\varphi$ is a self-homeomorphism of $S$ with distinguished points $E$, if $\varphi$ is a self-homeomorphism of $S$ that maps the set of distinguished points $E$ to itself.
Notice that each self-homeomorphism of the punctured surface $S\setminus E$ extends to a self-homeomorphism of the surface $S$ with set of distinguished points $E$.
We will sometimes identify
self-homeomorphisms of $S\setminus E$ and self-homeomorphism of $S$ with set $E$ of distinguished points.

For a (connected oriented closed or punctured) surface $S$ and a finite subset $E$ of $S$ a finite non-empty set of mutually disjoint Jordan curves $\{ C_1 ,
\ldots , C_{\alpha}\}$ in $S\setminus E$ is called
admissible for $S$ with set of distinguished points $E$ if
it is admissible for $S \setminus E$. An admissible system of
curves for $S$ with set of distinguished points $E$ is said to reduce a mapping class $\mathfrak{m}$ on $S$ with set of distinguished
points $E$, if the induced mapping class on $S\setminus E$ is reduced by this system
of curves.

\index{mapping class ! reducible} \index{mapping class !
irreducible}
Thurston calls a self-homeomorphism $\varphi$ of a closed or punctured surface $S$
reduced by the set $\{ C_1 , \ldots , C_{\alpha}\}$, if
this set is admissible and
$$
\varphi (C_1 \cup C_2 \cup \ldots \cup C_{\alpha}) = C_1 \cup C_2
\cup \ldots C_{\alpha} \, .
$$
A self-homeomorphism $\varphi$ of $S$ is called reducible if it is isotopic to a reduced mapping and irreducible
otherwise. A mapping class is reducible if it consists of
reducible maps. Similarly, a braid $b \in {\mathcal B}_n$ is called
reducible, if
its associated mapping class ${\mathfrak m}_b = \Theta_n(b)  \in
{\mathfrak M} (\overline{\mathbb D} ;
\partial \, {\mathbb D} , E_n^0)$ is reducible and is called
irreducible otherwise. A conjugacy class is called reducible if its
representatives are reducible. \index{conjugacy class of mappings ! reducible} \index{braid ! reducible} \index{braid ! irreducible}
\index{conjugacy class of mappings ! irreducible}

Bers \cite{Be1} gave a proof of Thurston's Theorem from the point of
view of Teichm\"{u}ller theory and obtained a description of reducible
mappings. The proof of our Main Theorem makes explicit use of the
technique developed by Bers. We will outline now the results of
Bers' approach to Thurston's theory which we need.


Consider a smooth closed or punctured surface $S$. We do not require that it is connected. We
allow $S$ to be the union of more than one, but at most finitely
many, connected components. A conformal structure on $S$ is a
homeomorphism $w$ of $S$ onto a Riemann surface. Let $\varphi$ be a
self-homeomorphism of $S$. Consider the extremal problem to find the
following infimum  \index{$I(\varphi)$}
$$
{I}(\varphi) \overset{\rm def}{=} \mbox{${\rm inf} \, \{ K (w \circ
\tilde\varphi \circ w^{-1}) : w : S \to w(S)$ is a conformal
structure on $S$,}
$$
\begin{equation}\label{eq2.30}
\mbox{$\tilde\varphi$ is free isotopic to $\varphi\}$.}
\end{equation}
This extremal problem differs from Teichm\"uller's extremal problem
by varying also the conformal structure. Notice that in \eqref{eq2.30} we do not
required that the conformal structures are quasiconformal. If the
infimum is realized on a pair $(w_0 , \varphi_0)$, then $w_0 \circ
\varphi_0 \circ w_0^{-1}$ is called absolutely extremal
\index{mapping ! absolutely extremal} and $w_0$ is called a
$\varphi_0$-minimal conformal structure.
\index{conformal structure}
\index{conformal structure ! $\varphi_0$-minimal
conformal structure}

Denote by $\mathfrak{m}_{\varphi}$ the mapping class of $\varphi$
and by $\widehat{\mathfrak{m}_{\varphi}}$  its conjugacy class,
$$
\widehat{\mathfrak{m}_{\varphi}} = \{ \psi = w \circ \tilde
{\varphi} \circ w^{-1}: w\; \mbox{is a conformal structure on} \;S,\;
\tilde{\varphi}\; \mbox{is isotopic to} \; \varphi \}.
$$
Then $I(\varphi)$ can be written as follows
\begin{equation}
I(\varphi)= {\rm {inf}} \{K(\psi): \psi \in
\widehat{\mathfrak{m}_{\varphi}}\}.
\end{equation}
In the following we will again consider connected surfaces $S$ unless said
otherwise. A self-homeomorphism $\varphi$ of $S$ is periodic if
$\varphi^n$ is the identity on $S$ for a natural number $n$. For
mappings that are isotopic to periodic self-homeomorphisms the
infimum is attained. More precisely, the following theorem holds.
\index{mapping ! periodic}

\begin{thm}[Bers, see \cite{Be1}]\label{thm2.8} A self-homeomorphism
$\varphi$ of a surface $S$ is (free) isotopic to a periodic
self-homeomorphism iff there is a conformal structure $w$ on $S$ and
a self-homeomorphism $\tilde\varphi$ of $S$ such that $w \circ
\tilde\varphi \circ w^{-1}$ is conformal (thus $K (w \circ
\tilde\varphi \circ w^{-1}) = 0$). Moreover, if $\varphi$ is free
isotopic to a periodic self-homeomorphism then there is a
$\varphi$-minimal conformal structure of first kind.
\end{thm}

\index{Bers}
For self-homeomorphisms of $S$ that are not isotopic to periodic
ones the following theorem holds.

\begin{thm}[Bers, see \cite{Be1}]\label{thm2.9.} A conformal structure $w$
of second kind on a surface $S$ cannot be $\varphi$-minimal for a
self-homeomorphism $\varphi$ of $S$ that is not free isotopic to a
periodic one. Moreover, there exists a conformal structure $w_1$ of
first kind and a self-homeomorphism $\tilde\varphi$ of $S$ which is
isotopic to $\varphi$ and such that $K (w_1 \circ \tilde\varphi
\circ w^{-1}_1) < K (w \circ \varphi \circ w^{-1})$.
\end{thm}

In the light of the two theorems it is sufficient to consider the
extremal problem only for conformal structures of first kind.
Moreover, we may fix a reference conformal structure of first kind
and replace $S$ by the obtained Riemann surface $X$. Composing the
mappings $w$ in \eqref{eq2.30} with the inverse of the reference conformal
structure we may consider conformal structures on $X$ rather than on
$S$. Assume that $X$ is of genus $g$ with $m$ punctures and $3g-3+m
> 0$. (We require that the universal covering of $X$ is ${\mathbb
C}_+$ and, in case of genus $0$, that the number of punctures is at
least $4$, to avoid trivial cases.)

The Teichm\"{u}ller space of a Riemann surface $\mathcal{T}(X)$ of first
kind can equivalently be described as follows. Consider the set of
\textit{all} homeomorphisms of $X$ onto another Riemann surface $Y$
of first kind (instead of \textit{quasiconformal} homeomorphisms).
Call two (not necessarily quasiconformal) homeomorphisms $w_j:X \to
Y_j,\; j=1,2,$ Teichm\"{u}ller equivalent if there is a conformal
mapping $c:Y_1 \to Y_2$ such that $w_2^{-1}\circ c \circ w_1$ is
isotopic to the identity. The thus obtained equivalence classes are
the same as the Teichm\"{u}ller classes. Indeed, if the Riemann surfaces
$X$ and $Y$ are of first kind and $w$ is an arbitrary homeomorphism
from $X$ onto $Y$ then $w$ can be extended to a homeomorphism
between closed Riemann surfaces. The extended homeomorphism can be
uniformly approximated by smooth, and thus, quasiconformal
homeomorphisms. For Riemann surfaces of first kind this implies also
that the Teichm\"{u}ller distance \eqref{eq2.18} can be computed by
letting all homeomorphisms in \eqref{eq2.18} be arbitrary
homeomorphisms instead of quasiconformal homeomorphisms. In
particular, for a Riemann surface $X$ of first kind and for
homeomorphisms $w_1 : X \to X_1$, $w_2 : X \to X_2$, the
Teichm\"uller distance $d_{\mathcal T} ([w_1],[w_2])$ can also be
written as follows:
$$
d_{\mathcal T} ([w_1],[w_2]) = {\rm inf} \, \biggl\{ \frac12 \log
K(g): g:X_1 \to X_2 \ \mbox{a surjective homeomorphism}
$$
\begin{equation}\label{eq2.31}
\mbox{which is isotopic to $w_2 \circ w_1^{-1}$ through such
homeomorphisms}\biggl\}.
\end{equation}

\noindent Let $X$ be a Riemann surface of first kind and let
$\varphi$ be a self-homeomorphism of $X$. Denote by $\varphi^*$
\index{$\varphi^*$} \index{modular transformation}  the modular
transformation induced by (the mapping class of) $\varphi$ on
${\mathcal T} (g,m)$. Put

\begin{equation}\label{eq2.32}
L(\varphi^*) \overset{\rm def}{=}  \underset{\tau \in {\mathcal T}
(g,m)}{\rm inf} d_{\mathcal T} (\tau , \varphi^* (\tau)) \, .
\end{equation}
The quantity $L(\varphi^*)$ is \index{translation length}
\index{$L(\varphi^*)$} called the translation length of $\varphi^*$.
Write $\tau = [w]\,$. Then $\varphi^*([w]) = [w \circ \varphi],\,$
and by \eqref{eq2.30} and \eqref{eq2.31}
\begin{equation}\label{eq2.33}
\frac12 \log I(\varphi) = \underset{\tau \in {\mathcal T} (g,m)}{\rm
inf} d_{\mathcal T} (\tau , \varphi^* (\tau)).
\end{equation}

Bers uses the following terminology in analogy to the classification
of elements of ${\rm PSL} (2,{\mathbb Z})$. Recall that ${\rm PSL} (2,{\mathbb Z})$ is the quotient of the group ${\rm SL} (2,{\mathbb Z})$ of
quadratic matrices of determinant $1$ with integer entries by the subgroup consisting of the identity ${\rm Id}$ and $-{\rm Id}$. \index{${\rm SL} (2,{\mathbb Z})$} \index{${\rm PSL} (2,{\mathbb Z})$}
Consider a modular
transformation $\varphi^*$, i.e. an element of the modular group
${\rm Mod}(g,m)$. \index{${\rm Mod}(g,m)$} A point $\tau \in
{\mathcal T} (g,m)$ is called $\varphi^*$-minimal, if $d_{\mathcal
T} (\tau , \varphi^* (\tau)) = L(\varphi^*)$. A modular
transformation $\varphi^* \in {\rm Mod} (g,m)$ is elliptic, if it
has a fixed point in ${\mathcal T} (g,m)$, parabolic, if it has no
fixed point but $L(\varphi^*) = 0$, hyperbolic, if $L(\varphi^*) >
0$ and $L(\varphi^*)$ is attained, and pseudohyperbolic, if $L
(\varphi^*) > 0$ but $d_{\mathcal T} (\tau , \varphi^* (\tau)) >
L(\varphi^*)$ for all $\tau \in {\mathcal T} (g,m)$. \index{$\varphi^*$-minimal}

\index{modular transformation ! elliptic} \index{modular
transformation ! hyperbolic} \index{modular transformation !
parabolic} \index{modular transformation ! pseudohyperbolic}

A conformal structure $w$ is $\varphi$-minimal for a
self-homeomorphism $\varphi$ of $X$ iff $[w]$ is
$\varphi^*$-minimal.

\begin{thm}[Bers, see \cite{Be1}]\label{thm2.10} An element $\varphi^* \in
{\rm Mod} (g,m)$ is elliptic iff it is periodic. This happens iff
the absolutely extremal map in the isotopy class of
self-homeomorphisms containing $\varphi$ is conformal.
\end{thm}

The following theorem is a reformulation of Thurston's result.

\begin{thm}[see \cite{Be1}]\label{thm2.11} For an irreducible
self-homeomorphism $\varphi$ of a Riemann surface $X$ of first kind
the modular transformation $\varphi^*$ of $\varphi$ is either
elliptic or hyperbolic.
\end{thm}

\begin{cor}\label{corr2.1} An irreducible
self-homeomorphism $\varphi$ of a Riemann surface $X$ of first kind
leads to an absolutely extremal self-homeomorphism $\tilde\varphi$
of a Riemann surface $Y$ by isotopy and conjugation with a
homeomorphism. For the quasiconformal dilatation $K(\tilde\varphi)$
of the absolutely extremal mapping $\tilde\varphi$ we have
\begin{equation}\label{eq2.33bis}
\frac12 \log K(\tilde\varphi) = \frac12 \log I( \varphi) =
L(\varphi^*).
\end{equation}
\end{cor}

The following theorem characterizes the absolutely extremal maps
with hyperbolic modular transformation in terms of Teichm\"{u}ller
mappings and quadratic differentials.

\begin{thm}[Bers, see \cite{Be1}]\label{thm2.12} Let $X$ be a Riemann
surface of genus $g$ with $m$ punctures, $3g-3+m > 0$. A mapping
$\varphi : X \to X$ is absolutely extremal, iff it is either
conformal or a Teichm\"uller mapping satisfying the following two
equivalent conditions
\begin{enumerate}
\item[(i)] {\it the mapping $\varphi \circ \varphi$ is also a Teichm\"uller mapping
with $K(\varphi \circ \varphi) = (K(\varphi))^2$,}
\item[(ii)] {\it the initial and terminal quadratic differentials of $\varphi$
coincide.}
\end{enumerate}
\end{thm}
\begin{rem}\label{rem2.1}
Suppose $\varphi$ is an absolutely extremal self-homeomorphism of $X$ with quadratic differential $\phi$. Then its modular transformation maps the Teichm\"uller disc ${\mathcal D}_{X,\phi}$ homeomorphically onto itself.
\end{rem}
Indeed,
the mapping $\varphi$ has Beltrami differential $k\frac{|\phi|}{\phi}$, and hence it has initial quadratic differential $\phi$. The inverse $\varphi^{-1}$ has Beltrami differential $-k\frac{|\psi|}{\psi}$, where $\psi$ is the terminal quadratic differential of $\varphi$. Since $\varphi$ is absolutely extremal, $\psi=\phi$.

For each $z\in \mathbb{D}$ and $\mu_z=z \frac{|\phi|}{\phi}$  the element $\varphi^*([W^{\mu_z}]) \in \mathcal{T}(X)$ is represented by  $W^{\mu_z}\circ \varphi$. Since $\varphi^{\pm 1}$ has Beltrami differential ${\pm} k\frac{\phi}{|\phi|}$ for a number $k\in (0,1)$,  the homeomorphism $\varphi^{\pm 1}$ has the form $W^{\mu_{\pm k}}$. Hence the composition $W^{\mu_z}\circ \varphi= W^{\mu_z} \circ (W^{\mu_{-k}} )^{-1}$  equals $W^{\mu_{z'}}$ for a number $z'\in\mathbb{D}$. Since also for the inverse $\varphi^{-1}$ the composition $W^{\mu_z}\circ \varphi^{-1}=W^{\mu_z} \circ (W^{\mu_{k}} )^{-1}   $ equals $W^{\mu_{z''}}$ for a number $z''\in\mathbb{D}$,
$\varphi$ maps  the Teichm\"uller disc ${\mathcal D}_{X,\phi}$ homeomorphically onto itself.
In case
$z=x\in \mathbb{R}$ is a positive real number, the composition $W^{\mu_z}\circ (\varphi)^{\pm 1}$  is equal to  $W^{\mu_{z_{\pm}}}$ for real numbers $z_{\pm}$.
In other words, $\varphi^*$ maps real points $\{\mu_z\}, z \in \mathbb{R},$ in the Teichm\"uller disc to real points in the Teichm\"uller disc.

\medskip
The following definition of pseudo-Anosov mappings is equivalent to Thurston's definition (see \cite{Th}, \cite{FLP}).
\begin{defn}\label{defn2.2}
Let $\,S\,$ be an oriented smooth connected surface. A self-homeomorphism $\varphi$ of $S$ is called  a pseudo-Anosov mapping,
if there exists a homeomorphism $w:S\to X$ onto a Riemann surface $X$ of first kind such that $\,w\circ \varphi\circ w^{-1} $ is
an absolutely extremal self-homeomorphism of $X$ with hyperbolic modular transformation
(equivalently, a non-periodic absolutely extremal self-homeomorphism of $X$ ).
\end{defn}

The absolutely extremal self-homeomorphisms $\varphi$ of a Riemann surface of first kind with
hyperbolic modular transformation (equivalently, the non-periodic  absolutely extremal self-homeomorphisms $\varphi$ of a Riemann surface of first kind)  are characterized by the following properties. There is a  quadratic differential
$\phi$ such that $\varphi$ maps singular points to singular points,
it maps the leaves of the horizontal foliation to leaves of the
horizontal foliation and leaves of the vertical foliation to leaves
of the vertical foliation. These two foliations are orthogonal
outside the common singularity set and are measured foliations in the sense of Thurston by \index{foliation ! measured}
using the metric $ds = \vert\phi\vert^{\frac12}$ to measure the
distance between leaves. The mapping $\varphi$ decreases the
distance between horizontal trajectories by the factor
$K^{-\frac12}$ and increases the distance between vertical
trajectories by the factor $K^{\frac12}$. Here $K$ is the
quasiconformal dilatation of $\varphi$.
\index{mapping ! pseudo-Anosov}

The image of $(-1,1)$ under an isometry from $(-1,1)$  with metric $\frac{dx}{1-x^2}$ into the Teichm\"uller space ${\mathcal T} (g,m)$ with Teichm\"uller metric is called a geodesic line.
\index{geodesic ! line}
Each pair of distinct points $\tau_1, \tau_2 \in \mathcal{T}(g,m)$ lies on a unique geodesic line. This geodesic line contains
each point $\tau$ for which $d_{\mathcal{T}}(\tau_1,\tau) + d_{\mathcal{T}}(\tau,\tau_2)= d_{\mathcal{T}}(\tau_1,\tau_2)$ (see \cite{Be1}, \cite{Krav}).

Self-homeomorphisms of Riemann surfaces with parabolic or
pseudohyperbolic modular transformation are reducible. Notice that
self-homeomorphisms with elliptic modular transformation also may be
reducible but those with hyperbolic modular transformation are
irreducible (see Theorem \ref{thm2.17}, statement ($3$) below).

Following Bers \cite{Be1} we consider now the extremal problem for
the quasiconformal dilatation in the case of reducible
self-homeomorphisms of Riemann surfaces.

Suppose again $X$ is a Riemann surface of first kind with $3g-3+m >
0$ and a self-homeomorphism $\varphi : X \to X$ is reduced by a
non-empty admissible system of curves $\{ C_1 , \ldots ,
C_{\alpha}\}$. $\varphi$ is called \index{mapping ! maximally
reduced} maximally reduced by this system if there is no admissible
system of more than $\alpha$ curves which reduces a mapping which is
isotopic to $\varphi$. $\varphi$ is called completely reduced by
this system of curves if for each connected component $X_j$ of the
complement $X \backslash \underset{\ell =
1}{\overset{\alpha}{\bigcup}} C_{\ell}$ and the smallest positive
integer $N_j$, for which $\varphi^{N_j} (X_j)=X_j$, the  map
$\varphi^{N_j} \mid X_j$ is irreducible.

\begin{lemm}[\cite{Be1}]\label{lemm2.14}
A reducible mapping is isotopic to a maximally reduced mapping.
\end{lemm}

\begin{lemm}[\cite{Be1}]\label{lemm2.15} If $\varphi$ is
maximally reduced by a system of curves $\{ C_1 , \ldots ,
C_{\alpha}\}$ it is completely reduced by this system.
\end{lemm}

\index{mapping ! completely reduced} Let $\varphi$ be a
self-homeomorphism of a Riemann surface $X$ of first kind, $3g-3+m >
0$, which is completely reduced by a non-empty admissible system $\{
C_1 , \ldots , C_{\alpha}\}$ of curves. If $\varphi$ is not periodic
then an absolutely extremal self-mapping of a Riemann surface
related to $\varphi$ by isotopy and conjugation does not exist (see
Theorem \ref{thm2.17} below). But there exists an absolutely
extremal self-mapping $\psi$ of a {\it nodal} Riemann surface $Y$.
The nodal Riemann surface $Y$ is the image of $X$ by a continuous
mapping $w$ which collapses each curve $C_j$ to a point and maps
$X \, \backslash \,
\underset{1}{\overset{\alpha}{\bigcup}} \ C_j$
homeomorphically onto its image. The absolutely
extremal mapping $\psi$ is related to $\varphi$ by isotopy on $X$
and semi-conjugation with $w$. The nodal Riemann \index{Riemann
surface ! nodal} surface $Y$ and the absolutely extremal mapping
$\psi$ on it can be regarded as a ``limit'' of a sequence of
non-singular Riemann surfaces $X_j$ and self-homeomorphisms
$\varphi_j$ of $X_j$. For the $X_j$ we have $X_j = w_j (X)$ for a
quasiconformal complex structure $w_j$ on $X$. The $\varphi_j$ are
related to $\varphi$ by isotopy and conjugation: $w_j^{-1} \circ
\varphi_j \circ w_j$ is isotopic to $\varphi$ on $X$. The
quasiconformal dilatations $K(\varphi_j)$ converge to the infimum
$${\rm inf} \, \{ K(w \circ \tilde\varphi \circ w^{-1}) : w \in
QC(X), \ \tilde\varphi \ \mbox{isotopic to} \ \varphi\} =
e^{2L(\varphi^*)}$$
and are strictly larger than the infimum.


We will need the details later and give them here. A nodal Riemann
\index{Riemann surface ! nodal} surface $X$ (or Riemann surface with
nodes) is a one-dimensional complex space, each point of which has a
neighbourhood which is either biholomorphic to the unit disc
${\mathbb D}$ in the complex plane or to the set $\{ z = (z_1, z_2)
\in {\mathbb D}^2 : z_1 z_2 = 0\}$. (A mapping on the latter set is
holomorphic if its restriction to either of the sets, $\{ (z_1,0) :
z_1 \in {\mathbb D}\}$, and $\{ (0,z_2) : z_2 \in {\mathbb D}\}$, is
holomorphic.) In the second case the point $(0,0)$ is called a node. We
assume that $X$ is connected and has finitely many nodes. Let
${\mathcal N}$ be the set of nodes. The connected components of $X
\backslash {\mathcal N}$ are called the parts of the nodal Riemann
surface. We do not require that the set ${\mathcal N}$
\index{$\mathcal{N}$} of nodes \index{node} is non-empty. If it is
empty we also call the Riemann surface non-singular. Thus,
non-singular Riemann surfaces are particular cases of nodal Riemann
surfaces.

We will say that a non-singular Riemann surface is of finite type
\index{Riemann surface ! finite type and stable} and stable if it
has no boundary continuum, has genus $g$ and $m$ punctures with $2g
-2+m
> 0$ (hence the universal covering is
${\mathbb C}_+$).
A connected nodal Riemann surface with finitely many parts, each of
which is a stable Riemann surface of finite type, is called of
finite type and stable.


Let $X$ and $Y$ be stable nodal Riemann surfaces of finite type. A
surjective homeomorphism $\varphi : X \to Y$ is orientation
preserving if its restriction to each part of $X$ is so. Notice that
$\varphi$ defines a bijection between the parts of $X$ and the parts
of $Y$. The quasiconformal dilatation of $\varphi$ is defined as
\begin{equation}\label{eq2.34}
K(\varphi) = \max_{X_j} \ K (\varphi \mid X_j) \, ,
\end{equation}
where $X_j$ runs over all parts of $X$.


Let $\varphi$ be an orientation preserving self-homeomorphism of a
stable nodal Riemann surface $X$ of first kind. The mapping
$\varphi$ permutes the parts of $X$ along cycles. Let $X_0$ be a
part of $X$ and let $n$ be the smallest number for which $\varphi
^n(X_0) = X_0$. Put $X_j = \varphi^j (X_0)$ for $j=1,\ldots , n-1,$,
and call $(X_0, X_1, ... , X_{n-1})$ a $\varphi$-cycle of length
$n$. The mapping $\varphi$ is called absolutely extremal if for any
$\varphi$-cycle $(X_0, X_1, ... , X_{n-1})$
the restriction $\varphi | X_0 \cup ... \cup X_{n-1}$ to the Riemann
surface $X_0 \cup ... \cup X_{n-1}$ (which is not connected if
$n>1$) is absolutely extremal. In other words the following holds.
Let $w : X \to Y$ be a homeomorphism onto another nodal Riemann
surface $Y$ (considered as conformal structure on $X$). Let
$\hat\varphi$ be a self-homeomorphism of $X$ which is isotopic to
$\varphi$. Let $(X_0, ... , X_{n-1})$ be a cycle of parts for
$\varphi$. Then
\begin{equation}\label{eq2.35}
\underset{0 \leq j \leq n-1}{\max} \ K (\varphi \mid X_j) \leq
\underset{0 \leq j \leq n-1}{\max} \ K (w \circ \hat\varphi \circ
w^{-1} \mid w(X_j)) \, .
\end{equation}

If the mapping $\varphi | X_0 \cup ... \cup X_{n-1}$ is absolutely extremal, then the restriction $\varphi ^n| X_0$ is an absolutely extremal self-homeomorphism of the connected Riemann surface $X_0$.
Notice that if $\varphi$ fixes all parts of $X$ then \eqref{eq2.35}
is equivalent to the condition that $\varphi \mid X_j$ is an
absolutely extremal self-homeomorphism of $X_j$ for each part $X_j$ of
$X$. In this case Theorem \ref{thm2.11} applied to the parts gives a
description of the absolutely extremal self-homeomorphisms.

Let $X_0$, $X_1,\ldots X_n$ be Riemann surfaces, and let $\hat \varphi$ be a self-homeomorphism of the disjoint union $X\stackrel{def}=X_0 \cup ...  \cup
X_{n-1}$ that commutes the $X_j$ along the cycle
\begin{align}\label{eq2.14b}
X_0 \overset{\hat\varphi}{-\!\!\!-\!\!\!\longrightarrow} X_1
\overset{\hat\varphi}{-\!\!\!-\!\!\!\longrightarrow} \ldots
\overset{\hat\varphi}{-\!\!\!-\!\!\!\longrightarrow} X_n \overset{\rm
def}{=} X_0 .
\end{align}
Let $w$ be a conformal structure on $X$, in other words, consider
Riemann surfaces $Y_j$, their disjoint union $Y\stackrel{def}= \bigcup_{j=0}^{n-1} Y_j$, and a conformal mapping $w:X\to Y$.
We may assume that $w$ maps $X_j$ conformally onto $Y_j$ for each $j$.
We define $\psi =  w \circ \hat \varphi
\circ  w ^{-1}$ on $Y = w (X)$. Let $\hat \varphi _j = \hat \varphi
|X_j$, $\psi_j = \psi |Y_j$ for $j=0,\ldots,n-1,\,$ and $w_j=w \mid
X_j,\, j=0,\ldots, n-1,\,$ $X_n=X_0$, and $ w _n = w | X_n =  w _0$.
Then there is following commutative  diagram

\begin{figure}[H]
\begin{center}
$$
\xymatrix{
X_0 \ar[rr]^{\widehat\varphi} \ar[d]^{ w_0} &&X_1 \ar[rr]^{\widehat\varphi} \ar[d]^{ w_1} &&\ldots \ar[rr]^{\widehat\varphi} &&X_{n-1} \ar[rr]^{\widehat\varphi} \ar[d]^{ w_{n-1}} &&X_n = X_0 \ar[d]^{ w_n =  w_0} \\
Y_0 \ar[rr]^{\psi} &&Y_1\ar[rr]^{\psi} &&\ldots \ar[rr]^{\psi}
&&Y_{n-1} \ar[rr]^{\psi} &&Y_n = Y_0 }\,.
$$
\end{center}
\caption{A commutative diagram for self-homeomorphisms of non-connected surfaces}
\label{fig2.4}
\end{figure}
The commutativity of the diagram implies the following equations
\begin{align}\label{eq2.36}
\psi_j \quad\; = & w _{j+1} \circ \hat \varphi_j \circ  w_j ^{-1}, \; j= 0,
\ldots n-1\,,\nonumber\\
\psi ^j |Y_0 = & w _{j} \circ \hat \varphi^j \circ  w _0^{-1}, \; j=
1, \ldots n\nonumber\,,\\
\psi ^n |Y_0 = & w _0 \circ \hat \varphi^n \circ  w
_0^{-1}\,.
\end{align}

It will be convenient to have in mind the following lemma which
describes the absolutely extremal self-homeomorphisms on the cycles
of parts of the nodal surface $X$.

\begin{lemm}\label{lemm2.16} Let $X_0, X_1, ... , X_{n-1}, X_n
\overset{\rm def}{=} X_0$ be non-singular stable Riemann surfaces of
finite type. Suppose $\varphi $ is a self-homeomorphism of $X_0 \cup
... \cup X_{n-1}$ which permutes the $X_j$ along the $n$-cycle
\begin{align}\label{eq2.35''}
X_0 \overset{\varphi}{-\!\!\!-\!\!\!\longrightarrow} X_1
\overset{\varphi}{-\!\!\!-\!\!\!\longrightarrow} \ldots
\overset{\varphi}{-\!\!\!-\!\!\!\longrightarrow} X_n \overset{\rm
def}{=} X_0 .
\end{align}
\noindent Then $\varphi $ is absolutely extremal iff the following
two conditions hold.
\begin{itemize}
\item [$(1)$] The mapping $F \overset{\rm def}{=} \varphi ^n | X_0$ is
absolutely extremal.
\item [$(2)$]  The quasiconformal dilations $K(\varphi_j)$ of the mappings $\varphi_j\stackrel{def}= \varphi|X_j$
satisfy the equality
\begin{equation}\label{eq2.35'}
\frac{1}{2}\log K(\varphi) = \max _j \frac{1}{2} \log K(\varphi_j) =   \frac{1}{2n} \log
K(\varphi ^n \mid X_0)\,.
\end{equation}
\end{itemize}
\smallskip

\noindent Moreover, if $\varphi $ is absolutely extremal then one of the following two situations occurs.
\begin{itemize}
\item [$(2a)$] $\;$ The mapping $\,F\,$ is conformal and $\,\varphi\,$ is
conformal.
\item [$(2b)$]$\;$The mapping $\,F: X_0\toitself \,$ is a Teichm\"uller mapping whose initial
and terminal quadratic dfferentials are equal,
and the Teichm\"{u}ller
classes $\,[\varphi ^j] \in \mathcal{T}(X_0)\,,$ $\,j= 1, ... , n-1,\,$
divide the bounded
segment  with endpoints $[id]$ and $[F]$ on the unique geodesic line through these two points into $n$
segments, each of which has $d_{\mathcal{T}}$-length equal to
$$
\frac{1}{2n} d_{\mathcal{T}(X_0)}([id],[F])= \frac{1}{2n} \log K(F)\,.
$$
Moreover, for each $j,\; j=1, ... , n-1,$  the mapping $\varphi_j\stackrel{def}= \varphi|X_j$   is a Teichm\"{u}ller map from
$\varphi^j(X_0)$ onto $\varphi^{j+1}(X_0)$.
\end{itemize}
\end{lemm}

\bigskip

\noindent {\bf Proof.}
Notice first that for all homeomorphisms $\hat \varphi$ that permute the $X_j$ along the cycle \eqref{eq2.35''}
the inequalities
\begin{align}\label{eq2.36x}
n \max_j \frac{1}{2} \log K(\hat{\varphi}_j)\geq \sum_0^{n-1}  \frac{1}{2} \log K(\hat{\varphi}_j)\geq  \frac{1}{2} \log K(\hat{\varphi}^n| X_0)\,
\end{align}
hold.

We prove first that $\varphi$ is absolutely extremal if $F=\varphi^n|X_0$ is absolutely extremal and
equality \eqref{eq2.35'} holds.
Let $\hat\varphi $ be a self-homeomorphism
of $X_0 \cup \ldots  \cup X_{n-1}$ that is isotopic to $\varphi$, and let $\psi$ be related to $\varphi$
by the diagram Figure \ref{fig2.4}. If  $F=\varphi^n|X_0$ is absolutely extremal, then
\begin{align*}
\frac{1}{2} \log K(\varphi^n| X_0)\leq \frac{1}{2} \log K(\psi^n| X_0)\,.
\end{align*}
If also equality \eqref{eq2.35'} holds, then
\begin{align*}
\frac{n}{2} \log K(\varphi)= \frac{1}{2} K(\varphi^n| X_0)  \leq \frac{1}{2} K(\psi^n| X_0)
\leq n \max_j \frac{1}{2} \log K(\psi_j)=\frac{n}{2} \log K(\psi)\,.
\end{align*}
We proved that $\varphi$ is absolutely extremal.

Vice versa, suppose $\varphi$ is absolutely extremal. Then $F=\varphi^n|X_0$ is absolutely extremal. 
It remains to show that condition (2) holds.
If $F$ is conformal, we associate to $F$ the self-homeomorphism
$$
X_0 \overset{\rm Id}{-\!\!\!-\!\!\!\longrightarrow} X_0
\overset{\rm Id}{-\!\!\!-\!\!\!\longrightarrow} \ldots
\overset{F}{-\!\!\!-\!\!\!\longrightarrow} X_0
$$
of the formal disjoint union of $n$ copies of $X_0$. This homeomorphism is conformal, hence its quasiconformal dilation equals zero.
By the Lemma on Conjugation this homeomorphism is conjugate to $\varphi$. Since $\varphi$ is absolutely extremal, $\varphi$ must
have vanishing quasiconformal dilation, hence $\varphi$ is  conformal (see equation \eqref{eq2.35}).

Consider the remaining case when $F$ is a Teichm\"uller mapping whose initial and terminal quadratic differentials are equal.
Let $\phi$ be the quadratic differential of $F$
and $\mathcal{D}_{X_0,\phi}$ the Teichm\"uller disc in $X_0$ corresponding to $\phi$. Suppose $\mu_{\phi}=k\frac{|\phi|}{\phi}$. The set $\Big\{ \{x\frac{|\phi|}{\phi}\}: x\in \mathbb{R}\Big\}$, equipped with the metric $d_{\mathcal{T}}$, is the unique geodesic line through $[F]=\{k\frac{|\phi|}{\phi}\}$ and $[{\rm Id}]= \{0\}$. Consider the points $\tau_j= \{x_j\frac{|\phi|}{\phi}\}$ on this line that divide the segment $\Big\{ \{x\frac{|\phi|}{\phi}\}: x\in [0,k]\Big\}$ on this line into $n$ segments of equal $d_{\mathcal{T}}$-length $\frac{1}{2n} \log K(F)$. The class $\tau_j,\, j=1,\ldots,n-1,$ is represented by the absolutely extremal self-homeomorphism $W^{\mu_{x_j}}$ of $X_0$, $\tau_0$ is represented by the identity $\rm Id$ and $\tau_n$ is represented by $W^{\mu_{x_n}}=W^{\mu_{k}}=F$.
Put $\mathring{X}_0=\mathring{X}_n=X_0$ and  $\mathring{X}_j= W^{\mu_{x_j}}(X_0),\, j=1,\ldots,n-1$.
By Lemma 9.1 of \cite{Let} for each $j=0,\ldots,n-1,$ the homeomorphism $\mathring{\varphi}_j \stackrel{def}= W^{\mu_{x_{j+1}}} \circ  (W^{\mu_{x_{j}}})^{-1}:\mathring{X}_j\to \mathring{X}_{j+1}$ is a Teichm\"uller mapping with quadratic differential $\phi$. For its quasi-conformal dilatation the equality $\frac{1}{2} \log K(\mathring{\varphi}_j)=d_{\mathcal{T}}(\{x_j\frac{|\phi|}{\phi}\},\{x_{j+1}\frac{|\phi|}{\phi}\})= \frac{1}{2n}\log K(F)$ holds.

Consider the mapping $\mathring{\varphi}$ on $\mathring{X}\stackrel{def}=\bigcup_0^{n-1}\mathring{X}_j$ that equals $\mathring{\varphi}_j $ on $\mathring{X}_j$. Then $\mathring{\varphi}_0^n\mid  \mathring{X}_0=F$, $\frac{1}{2}\log K(\mathring{\varphi})=\frac{1}{2n} \log K(\mathring{\varphi}^n\mid  \mathring{X}_0)$. Since $\mathring{X}_0=X_0$ and $\mathring{\varphi}^n\mid X_0= {\varphi}^n\mid X_0$ the mapping $\mathring{\varphi}$ on  $\mathring{X}$ is conjugate to the mapping ${\varphi}$ on $X$. Since ${\varphi}$ is absolutely extremal, the inequality
$$
\frac{1}{2}\log K({\varphi})\leq \frac{1}{2}\log K(\mathring{\varphi})=\frac{1}{2n}
\log K(\mathring{\varphi}^n\mid  \mathring{X}_0) =
\frac{1}{2n} \log K({\varphi}^n\mid {X}_0)
$$
holds. Condition $(2)$ follows from \eqref{eq2.36x}.

Conditions (1) and (2) imply that $\varphi$ is conformal if $F$ is conformal, and if $F$ is a Teichm\"uller mapping, then
equality \eqref{eq2.35'} implies that the elements $\tau_j$ of the
Teichm\"uller space represented by the $\varphi_j$
divide the segment on the geodesic line between $[Id]$ and $[F]$ into $n$ segments of equal length.
The lemma is proved. \hfill $\Box$

\smallskip

Now we describe in more detail the solution of the extremal problem
\eqref{eq2.30} in the reducible case. Let $X$ be a connected Riemann
surface, which is closed or of first kind, with universal covering
equal to $\mathbb{C}_+$. Let $\mathfrak{m} \in \mathfrak{M}(X)$ be
an isotopy class of orientation preserving self-homeomorphisms. Let
$\mathcal{C}$ be an admissible system of curves which completely
reduces an element $\varphi$ of $\mathfrak{m}$. By an isotopy we may
assume that $\mathcal{C}$ is real analytic (or even geodesic).
Associate to $X$ and the system of curves $\mathcal{C}$ a nodal
surface $Y$ and a continuous surjection $w:X \to Y$. This can be
done as follows.

Surround each connected component $\mathcal{C}_j$ of $\mathcal{C}$ by an annulus
$A(\mathcal{C}_j)$, which admits a conformal mapping $\mathfrak{c}_j$ onto a round annulus
$A_{r_j}=\{\frac{1}{r_j}< |z| < r_j\}\;\}$ for some $\; r_j>1,\; $ such that
$\mathfrak{c}_j$ maps the curve $\mathcal{C}_j$  to the unit circle. For all $j$ we let $\chi_j: [1,r_j]
\to [0,r_j]$ be a diffeomorphism such that $\frac{\chi_j(z)}{z}$ is the identity near $r_j$ and maps $1$ to $0$.

For each $j$ we define a continuous surjective mapping $w_j: A_{r_j}\to V_j\stackrel{def}= \{(z_1,z_2)\in \mathbb{C}^2:\, z_1 z_2 =0 ,\, |z_j|<r_j \, \mbox{for} \, j=1,2\}$ as follows.
For $|z|\leq 1$ we put
$w_j(z)= (\,0\,, \; \chi_j(|z|^{-1}) \cdot \frac{z}{|z|}) \in \mathbb{C}^2$ and for
$|z|\geq 1$ we put $w_j(z)= (\, \chi_j(|z|) \cdot
\frac{z}{|z|}\,,\,0\,)\in \mathbb{C}^2.$
The nodal surface $Y$ is obtained by gluing each $V_j$ to $X\setminus\bigcup \mathcal{C}_j$ along $A(\mathcal{C}_j)\setminus \mathcal{C}_j$,  using the homeomorphism $w_j\circ \mathfrak{c}_j:A(\mathcal{C}_j)\setminus \mathcal{C}_j\to V_j\setminus \{0\}$.
The set of nodes $\mathcal{N}$ of $Y$ is in bijective correspondence to the set of connected components of $\mathcal{C}$. The surjective mapping $w:X\to Y$ is defined
on $X\setminus \mathcal{C}$ so that it
takes this set identically to its copy in $Y$, and is equal to $w_j\circ \mathfrak{c}_j$ on each annulus $A(\mathcal{C}_j)$.
The thus defined mappings agree on the set
$\bigcup (A(\mathcal{C}_j) \setminus \bigcup \mathcal{C}_j)$, hence $w$ is correctly defined.

Since $\varphi$ maps the complement $X\setminus \bigcup_{C \in\mathcal{C}}C$ of the set of curves homeomorphically onto itself,
we may consider the mapping $w \circ \varphi \circ
w^{-1}$ on $Y \setminus \mathcal{N}$. It extends continuously to the
nodes. Denote the obtained mapping on $Y$ by
$\varphi_{\odot}$
and its isotopy class on $Y$ by
${\mathfrak{m}}_{\odot}$.
\index{$\varphi_{\odot}$}

Note that the nodal surface $Y$ is defined up to homeomorphism by the
isotopy class of $\mathcal{C}$. (We consider isotopies of the system of curves $\mathcal{C}$
within the class of real analytic systems of curves). The class
${\mathfrak{m}}_{\odot}$ is determined by $\mathfrak{m}$, and by the system of curves
$\mathcal{C}.\,$ The conjugacy class
$\widehat{{\mathfrak{m}}_{\odot}}$ of
${\mathfrak{m}}_{\odot}$ is defined by the isotopy class of
$\mathcal{C}$ and the class $\widehat{\mathfrak{m}}$.
\index{$\mathfrak{m}_{\odot}$ }
\index{$\widehat{{\mathfrak m}_{\odot}}$}

The mapping $\varphi_{\odot}$, and, hence its isotopy class ${\mathfrak{m}}_{\odot}$,
permutes the parts $Y_k$ of $Y \setminus \mathcal{N}$ along cycles denoted by $\mathring{{\rm cyc}}_j$.
\index{$\mathring{{\rm cyc}}_j$}
Let ${\mathfrak m}_{\odot,j}$ be \index{$\mathfrak{m}_{\odot,j}$}
the restrictions  of the class ${\mathfrak{m}}_{\odot}$ to the cycles $\mathring{{\rm cyc}}_j$.
We call the conjugacy classes $\widehat{\mathfrak{m}_{\odot,j}}$ of the restrictions to the cycles $\mathring{{\rm cyc}}_j$
the irreducible components of $\widehat{\mathfrak{m}}$
related to the class of $\mathcal{C}$. Notice \index{mapping class !
irreducible components of} that the class
$\widehat{{\mathfrak{m}}_{\odot}}$ determines the class
$\widehat{\mathfrak{m}}$ only modulo products of powers of Dehn
twists around curves which are homotopic to curves of the system
$\mathcal{C}$. It is not hard to see that the $\widehat{{\mathfrak{m}}_{\odot}}$ are irreducible (see alo Remark
\ref{rem6.1}).

Notice also that the isotopy class of a system of curves
$\mathcal{C}$ which completely reduces an element of $\mathfrak{m}$
is not uniquely determined by $X$ and $\mathfrak{m}$, even if we
require that the system maximally reduces the homeomorphism. In
particular, the type of the nodal surface $Y$ is not uniquely
determined. This may occur, for instance for reducible
homeomorphisms with elliptic modular transformation.

It is a remarkable fact that the extremal problem for the
quasiconformal dilatation has a solution in terms of nodal Riemann
surfaces and irreducible parts of mapping classes.

The following theorem is due to Bers.

\begin{thm}[\cite{Be1}]\label{thm2.17} Let $\mathfrak{m}$ be a mapping
class of orientation preserving self-homeomorphisms of a Riemann
surface $X$ of first kind with universal covering ${\mathbb C}_+$.
Let $\mathcal{C}$ be an admissible system of real analytic curves
which completely reduces an element $\varphi \in \,\mathfrak{m}$.
Choose a stable nodal Riemann surface $Y$ of first kind with set of
nodes $\mathcal{N}$ and a continuous surjection $w: X \to Y$ which
contracts each curve of $\mathcal{C}$ to a point and whose
restriction to the complement of $\mathcal{C}$ is a homeomorphism
onto $Y \setminus \mathcal{N}$. Denote by
${\mathfrak{m}}_{\odot}$ the mapping class on $Y$ induced by
$\mathfrak{m}$ and $w$, and by $\widehat
{{\mathfrak{m}}_{\odot}}$ its conjugacy class. Then the
following holds.
\begin{enumerate}
\item [$1$.] There exists a conformal structure $\tilde w$ on $Y,\,$ $\tilde w:
Y \to \tilde w(Y) = \tilde Y,\,$ and an absolutely extremal
self-homeomorphism $\tilde {\varphi}$ of $\tilde Y$, representing
the class $\widehat {{\mathfrak{m}}_{\odot}}$.
\end{enumerate}
\noindent The self-homeomorphism $\tilde {\varphi}$ has the
following stronger extremal properties.
\begin{enumerate}
\item [$2$.] The equality
\begin{align}\label{eq2.40}
I(\varphi) = e^{2L(\varphi^*)} = K (\tilde {\varphi})
\end{align}
holds for the modular transformation ${\varphi}^*$ of $\varphi$. Moreover, for  any continuous surjection
$w':X \to Y'$ onto a nodal Riemann surface $Y'$, such that the preimages of the nodes are disjoint
Jordan curves and the restriction of $w'$ to the complement of the
curves is a homeomorphism, and any self-homeomorphism $\varphi'$ in
the class ${\mathfrak{m}}'_{\odot}$ induced by
$\mathfrak{m}$ and $w'$, we have
$$
K(\tilde {\varphi}) \leq K(\varphi').
$$
\item [$3$.] If $\mathfrak{m}$ is reducible and not periodic then for each
non-singular Riemann surface $Y'$, each surjective homeomorphism
$w' : X \to Y'$, and each self-homeomorphism $\varphi'$ of $Y'$
such that $(w')^{-1} \circ \varphi' \circ w' \in \mathfrak{m}$, we
have strict inequality
$$
K(\tilde {\varphi}) < K(\varphi').
$$
\item [$4$.] If $\mathfrak{m}$ is reducible then there exists a sequence $Y^{(j)}$ of
non-singular Riemann surfaces $Y^{(j)}$, surjective homeomorphisms $w^{(j)}
: X \to Y^{(j)}$, and self-homeomorphism $\varphi^{(j)}$ of $X^{(j)}$ with the
following property:

$$
\mbox{ $(w^{(j)})^{-1} \circ \varphi^{(j)} \circ w^{(j)} \in \mathfrak{m}$ and
$K(\varphi^{(j)}) \to K(\tilde {\varphi})$.}
$$
\end{enumerate}
\end{thm}


Recall that a self-homeomorphism $\tilde  {\varphi}$ of a nodal Riemann surface $\tilde Y$
is absolutely extremal if the quantity $K(\tilde  {\varphi})=\max_{{\tilde Y}_j} K(\tilde  {\varphi}|Y_j)$ is smallest among the respective quantities for all nodal Riemann surfaces that are homeomorphic to $Y$ and all
self-homeomorphisms of $Y$ that are obtained from $\tilde  {\varphi}$ by isotopy and conjugation. Here ${\tilde Y}_j$ are the parts of $\tilde Y$. Recall also that the quantity $I({\varphi})$ for a self-homeomorphism $\varphi$ of a non-singular Riemann surface equals the infimum of $K(\varphi')$ over all self-homeomorphisms $\varphi'$ of non-singular Riemann surfaces
in the class of $\varphi$. Equality \eqref{eq2.40} relates the value $I({\varphi})$
to the quasiconformal dilatation of the absolutely extremal self-homeomorphism of the associated nodal surface.

The second part of
Statement (2)
is strictly stronger than statement (1). In Statement (1)
the nodal surface up to homeomorphism and the class of $\tilde\varphi$ are fixed. In statement (2) the
type of the nodal surface $Y'$ is not prescribed. It may be
different from the type of $Y$. We require only that $Y'$ is the image of $X$ under a continuous surjection that maps some subcollection of the set of admissible curves to nodes.

The deepest part of the theorem is its particular case concerning
irreducible homeomorphisms and Statement (3).

\chapter{The entropy of surface homeomorphisms}

\label{chapter-entropy}
\setcounter{equation}{0}

Here we give a self-contained proof of the Theorem of Fathi and Shub
on the entropy minimizing property of pseudo-Anosov self-homeomorphisms of closed Riemann surfaces of genus at least two following mainly along the lines of the original proof. A proof of the theorem for the case of punctured surfaces is also included. The exposition contains a thorough account on the necessary prerequisites on the trajectories of quadratic differentials.  The present proof of the Fathi and Shub Theorem on the relation between the entropy and the quasi-conformal dilatation is new. The chapter concludes with a thorough treatment of the entropy of self-homeomorphisms of Riemann surfaces of second kind which is a prerequisite for the treatment of reducible braids.
\section{The topological entropy 
of mapping classes}
\label{sec:entropy.1}

The topological entropy of continuous surjective mappings of a
compact topological space to itself is defined as follows.
\index{entropy}

\smallskip

Let $X$ be a compact topological space and $\varphi$ a continuous
mapping from $X$ onto itself. Let ${\mathcal A}$
\index{$\mathcal{A}$} \index{$\mathcal{B}$}
be a collection of
open subsets of $X$ which cover $X$ (for short, ${\mathcal A}$ is an
open cover of $X$). \index{cover}
For two open covers ${\mathcal A}$ and
${\mathcal B}$ we define ${\mathcal A} \vee {\mathcal B} = \{ A \cap B : A
\in {\mathcal A} , B \in {\mathcal B} \}$. \index{$\mathcal{A} \vee
{\mathcal B}$} Let ${\mathcal N} ({\mathcal A})$ be \index{$\mathcal
{N} ({\mathcal A})$} the minimal cardinality of a subset ${\mathcal
A}_1$ of ${\mathcal A}$ which is a cover of $X$. The entropy
$h(\varphi , {\mathcal A})$ of $\varphi$ with respect to ${\mathcal
A}$ is defined as
\begin{equation}\label{eq3.1}
h(\varphi , {\mathcal A}) \overset{\rm def}{=} \ \overline{\lim}_{N
\to \infty} \ \frac1N \log\ {\mathcal N} ({\mathcal A} \vee \varphi^{-1}
({\mathcal A}) \vee \ldots \vee \varphi^{-N} ({\mathcal A})) \, .
\end{equation}
(Here $\varphi^{-1} ({\mathcal A}) = \{ \varphi^{-1} (A) : A \in
{\mathcal A}\}$, $\varphi^{-k-1} ({\mathcal A}) = \varphi^{-1}
(\varphi^{-k} ({\mathcal A}))$.) The entropy $h(\varphi)$ of
\index{$h(\varphi)$} $\varphi$ is defined as
\begin{align}\label{eq3.1a}
h(\varphi) \stackrel{def}=\underset{\mathcal
A}{\sup} \ h (\varphi , {\mathcal A}).
\end{align}
\index{$h(\varphi , {\mathcal A})$}
An open cover ${\mathcal A}$ is a refinement of an open cover
${\mathcal B}$ if each set $A \in \mathcal{A}$ is contained in a set
$B \in \mathcal{B}$. We write $\mathcal{A} \prec \mathcal{B}$. A sequence $\{\mathcal{A}_n\}, \, n=1,2,\dots,$
of open covers is refining if for each $n$ the cover
$\mathcal{A}_{n+1}$ is a refinement of $\mathcal{A}_n$ and each open
cover $\mathcal{B}$ has a refinement among the $\mathcal{A}_n$. For
a refining sequence of open covers $\mathcal{A}_n$ the equality
$$
h(\varphi)=\lim_{n \to \infty}h(\varphi,\mathcal{A}_n)
$$
holds. (The sequence of numbers on the right is non-decreasing.) See
\cite{AKM}, Proposition 12.

An example of a refining sequence of covers is the following. Suppose $X$ is equipped with a metric. Let $\mathcal{A}_{n}$ be the cover consisting of all open balls of radius $\varepsilon_n$. If $\varepsilon_n$ decreases to zero for $n \to \infty$ then the sequence is refining.

\smallskip

\noindent The entropy is a conjugacy invariant (see Theorem 1 of \cite{AKM}):\\
{\it If $\psi : X \to Y$ is a homeomorphism of
topological spaces, then} $h(\psi \circ \varphi \circ \psi^{-1}) = h
(\varphi)$.
\smallskip

\noindent Let $\varphi$ be a self-homeomorphism of $X$. The following relation is not difficult to prove (see Theorem 2 of \cite{AKM}).\\
{\it For any non-zero integral number $n$}, $h(\varphi^n)= |n| \, h (\varphi)$.

\smallskip

\noindent Moreover, the following result of \cite{AKM} (see Theorem 4 there) holds.\\
{\it Let $X_1$ and $X_2$ be two closed subsets of a topological space $X$ such that $X=X_1 \cup X_2$. If a self-homeomorphism $\varphi$  of $X$ fixes each $X_j$, $j=1,2,$ setwise, then} $h(\varphi)= \max_{j=1,2}h(\varphi\mid X_j) $.\\
\noindent For more details see \cite{AKM}.

\smallskip

\noindent We als need the following property of entropies (see Theorem 5  of \cite{AKM}). \\
{\it Suppose $X$ is a compact topological space and $\sim$ is an equivalence relation on $X$.
Let $p:X\to X\diagup \sim$ be the projection of $X$ to the quotient. If $\tilde{\varphi}$ is a continuous mapping from $X$ into itself such that $\tilde{\varphi}(x)\sim \tilde{\varphi}(y)$ if $x\sim y$, then for the mapping ${\varphi}$ on the quotient that is defined by ${\varphi}\circ p= p\circ \tilde{\varphi}$ the inequality
$h({\varphi})\leq h(\tilde{\varphi})$ holds.}

We will be concerned with the topological entropy of \index{entropy
! of self-homeomorphisms} self-homeomorphisms of compact surfaces
(with or without boundary). Let $X$ be a compact surface (with or
without boundary) with a finite set $E_n$ of distinguished points.

\smallskip

Let ${\mathfrak m} \in {\mathfrak M} (X ; \partial X , E_n)$ (or
${\mathfrak m} \in {\mathfrak M} (X ; \emptyset , E_n)$ if $\partial X =\emptyset$)  be
a mapping class. The entropy of the mapping class \index{entropy !
of mapping classes} ${\mathfrak m}$ is defined as follows
\begin{equation}\label{eq3.2}
h({\mathfrak m}) = {\rm inf} \, \{ h(\varphi) : \varphi \in
{\mathfrak m} \} \, .
\end{equation}
Since the topological entropy of a homeomorphism is invariant under
conjugation the following holds:
\begin{equation}\label{eq3.2a}
h({\mathfrak m}) = h(\widehat {\mathfrak m}) = {\rm inf} \, \{ h (\psi
\circ \varphi \circ \psi^{-1}) : \varphi \in {\mathfrak m} \, , \
\psi : X \to Y \ \mbox{a homeomorphism}\} \, .
\end{equation}
For a braid $b \in \mathcal{B}_n$ with base point $E_n$ we define
the entropy $h(b)$ \index{entropy ! of braids} as the entropy of the
mapping class $\mathfrak{m}_b = \Theta_n(b) \in {\mathfrak M} (\bar
D ;
\partial D , E_n)$ corresponding to $b$:
\begin{equation}\label{eq3.41}
h(b) \overset{\rm def}{=} \ h(\mathfrak{m}_b) .
\end{equation}
By \eqref{eq3.2a} the entropy $h(b)$ does not depend on the choice
of the base point $E_n$ and is a conjugacy invariant. Hence,
\begin{equation}\label{eq3.42}
h(b) \, =\, h(\hat b)\, = \ h(\widehat {\mathfrak{m}_b})
\,\overset{\rm def}{=}\, {\rm inf} \, \{ h(\varphi) : \varphi \in
\widehat {\mathfrak m _b} \} \, .
\end{equation}
\index{entropy ! of conjugacy classes of braids} \index{$h(b)$} \index{$h(\hat{b})$}
The entropy of
self-homeomorphisms of closed surfaces was studied first in expos\'e~10 of the Asterisque volume dedicated to Thurston's work \cite{FLP}.

In Sections \ref{sec:entropy.2a} and \ref{sec:entropy.2} of this chapter we give the proof of the following two theorems.
\begin{thm}\label{thm3.2} Let $X$ be a (closed connected) Riemann surface
of genus $g$ with a set $E_m$ of $m\geq 0$ distinguished points, $3g-3+m>
0$. Let $\varphi_0 \in {\rm Hom} (X ; \emptyset , E_m)$ be a
non-periodic absolutely extremal self-homeomorphism of $X$ with set of distinguished
points $E_m$ (by an abuse of language we speak about a non-periodic absolutely extremal
self-homeomorphism of $X \backslash E_m$). Then
\begin{equation}\label{equation3.2'}
h(\varphi_0) = \frac{1}{2} \log K(\varphi_0) \, = \, L (\varphi_0^*) \,.
\end{equation}
\end{thm}

Here $K (\varphi_0)$ is the quasiconformal dilatation of
$\varphi_0$,
and $\varphi_0^*$ is the modular transformation on the
Teichm\"uller space ${\mathcal T} (g,m)$ induced by $\varphi_0$ and
$L(\varphi_0^*)$ is its translation length.

The theorem was first proved  for closed surfaces by Fathi and Shub in the exposé 10 of the volume \cite{FLP}.
The present proof of Theorem \ref{thm3.2} (see Section \ref{sec:entropy.3}) differs from that of Fathi and Shub and works also for punctured surfaces.
\index{Fathi} \index{Shub}

\begin{thm}[\cite{FLP}]\label{thm3.1} Let $X$ be a closed
connected Riemann surface of genus $g \geq 2$. Then any
pseudo-Anosov self-homeomorphism of $X$ is entropy minimizing in its
isotopy class.
\end{thm}

Theorem \ref{thm3.1} is proved in \cite{FLP} but the ingredients of the proof of Theorem \ref{thm3.1} in \cite{FLP} are spread over several chapters.
In Sections \ref{sec:entropy.2a} and \ref{sec:entropy.2} we present a self-contained proof of Theorem \ref{thm3.1} following mainly the ideas of Fathi and Shub.
Ahlfors' trick (see \cite{A2}, section 4, p.19-20), and the  Lemma \ref{lem3.4} below on the entropy of lifts to simple branched coverings are needed to prove the analog of Theorem \ref{thm3.1} for punctured Riemann surfaces. Theorem \ref{thm3.1} for punctured Riemann surfaces is proved in Section \ref{sec:entropy.3}. The result of Section \ref{sec:entropy.4} on the entropy of self-homeomorphisms of Riemann surfaces of second kind is needed to treat the entropy of reducible braids.

Recall that
a branched covering is called simple, if over each point there is at most one branch point and the mapping is a double branched covering in a neighbourhood of this branch point.
 \index{covering ! simple branched}

\begin{lemm}\label{lem3.4} Let $\varphi$ be a
self-homeomorphism of a compact smooth surface $X$. Let $\hat X
\overset{p}{\longrightarrow} X$ be a topological simple branched
covering of $X$ and let $\hat\varphi:\hat X \to \hat X$ be a self-homeomorphism of $\hat X$ such that $p \circ \hat\varphi = \varphi \circ p$ (for short, $\hat\varphi$ is a lift of $\varphi$). Then $h (\hat\varphi) = h(\varphi)$.
\end{lemm}

\noindent {\bf Proof of the Lemma \ref{lem3.4}.} We assume that the covering multiplicity $m$ is bigger than $1$. \\
{\bf I. The inequality $h(\varphi) \leq h (\hat\varphi)$} follows from Theorem~5 \cite{AKM}.
Here is a proof for convenience of the reader. Let ${\mathcal A}$ be
an open cover of $X$. Put $ p^{-1} ({\mathcal A})
\overset{\rm def}{=} \{ p^{-1} (A) : A \in {\mathcal A} \}$.
$p^{-1} ({\mathcal A})$ is an open cover of $\hat X$. Since $p \circ
\hat\varphi = \varphi \circ p$ it follows that $p^{-1} (\varphi^{-1}
({\mathcal A})) = \hat\varphi^{-1} (p^{-1} ({\mathcal A}))$. Further,
$p^{-1} ({\mathcal A} \vee {\mathcal B}) = p^{-1} ({\mathcal A})
\vee p^{-1} ({\mathcal B})$ for two open covers ${\mathcal A}$ and
${\mathcal B}$ of $X$. Hence
\begin{align}
{\mathcal N} \Big(p^{-1}({\mathcal A}) \vee \ldots \vee  \hat\varphi^{-N}(p^{-1} ({\mathcal A}))\Big) &=
{\mathcal N} \Big(p^{-1} ({\mathcal A}) \vee \ldots \vee p^{-1}
(\varphi^{-N} ({\mathcal A}))\Big) \nonumber \\
= {\mathcal N} \Big(p^{-1} ({\mathcal A} \vee \ldots \vee \varphi^{-N} ({\mathcal A}))\Big)\quad \quad
& ={\mathcal N}\Big({\mathcal A} \vee \ldots \vee \varphi^{-N}
({\mathcal A})\Big) \, . \nonumber
\end{align}
Hence
$$
h(\varphi) = \underset{{\mathcal A}}{\rm sup} \, h(\varphi ,
{\mathcal A}) = \underset{{\hat{\mathcal A} = p^{-1} ({\mathcal A})
\, {\rm for} \atop {\rm a \, cover} \, {\mathcal A} \, {\rm of} \,
X}}{\rm sup} \, h(\hat\varphi , \hat{\mathcal A}) \leq
\underset{{\tilde{\mathcal A} \, {\rm an \, arbitrary} \atop {\rm
cover \, of} \, \hat X}}{\rm sup} \, h (\hat\varphi , \tilde{\mathcal A})
= h(\hat\varphi) \, .
$$

Notice that whatever sequence of covers $\mathcal{A}_n$ of $X$ we take, the sequence of covers $p^{-1} ({\mathcal A}_n)$ is not refining if the covering multiplicity is bigger than $1$.

\noindent {\bf II. The opposite inequality $h(\varphi) \geq h (\hat\varphi)$.}
We will choose a suitable refining sequence of coverings of $X$ and associate to it a refining sequence of coverings of  $\hat X$. Equip $X$ with a metric $d$ and denote by $d$ also the induced metric on $\hat X$ (which is defined by putting the length of any smooth curve in $\hat X$ equal to the length of its projection). Let $\sigma >0$ be the minimal distance between two points in the branch locus of the covering.

A connected open subset $V$ of $X$ is called
nicely covered if $p^{-1}(V)$ contains no more than one branch point. For a
nicely covered connected open set $V$ the preimage $p^{-1}(V)$ consists either of $m$ or $m-1$ connected components.
In the first case $p$ is a homeomorphism from each connected component of $p^{-1}(V)$ onto $V$. In the second case
this statement is true for $m-2$ of the connected components of $p^{-1}(V)$.
The restriction of $p$ to the remaining component of $p^{-1}(V)$ is a double branched covering of $V$ with a single branch point.
Each open set of diameter not exceeding $\sigma$ is nicely covered.

\noindent {\bf Remarks preceeding the proof of the opposite inequality.}
Suppose we have three simply connected sets $A_1$, $A_2$, and $A_3$ belonging to a certain open cover $\mathcal A$ of $X$.
Suppose for integer numbers $k_1,k_2,$ and $k_3$ the set $A_1\cap \varphi^{-k_2}(A_2)\cap \varphi^{-k_3}(A_3)$ is not empty and the union $A_1\cup \varphi^{-k_2}(A_2)\cup \varphi^{-k_3}(A_3)$    is contained in a simply connected set $B$, that does not intersect the branch locus. Since $\varphi$ is a homeomorphism, none of the sets $A_j,\,j=1,2,3,$ intersects the branch locus. For $k'=1,2,3,$ we let $A_{k'}^{l}$ be the $m$ connected components of $p^{-1}(A_{k'})$,
and $\hat{\varphi}$ a self-homeomorphism of $\hat X$ for which
$p\circ \hat{\varphi}=\varphi\circ p$.
Then
the intersection of two or three sets among the $A_1^{j_1}$, $  \varphi^{-k_2}(A_2^{j_2})$, and
 $  \varphi^{-k_3}(A_2^{j_3})$ is not empty  if and only if the sets lie in the same of the $m$ connected components of $p^{-1}(B)$. Hence, there are exactly $m$  non-empty sets of the form
$A_1^{j_1}\cap  \varphi^{-k_2}(A_2^{j_2})\cap
  \varphi^{-k_3}(A_2^{j_3})$.

Consider the case when $A_1$ and $A_3$ do not meet the branch locus but $A_2$ contains a single point of the branch locus. Suppose the intersection  $A_1\cap \varphi^{-k_2}(A_2)\cap \varphi^{-k_3}(A_3)$ is not empty and the union $A_1\cup \varphi^{-k_2}(A_2)\cup \varphi^{-k_3}(A_3)$    is contained in a simply connected set $B$ with a single point in the branch locus. The previous arguments remain true for the $m-2$ connected components of
$p^{-1}(B)$ that do not contain a branch point. The connected component $B^{m-1}$, for which
$p:B^{m-1}\to B$ is a  double branched covering contains two preimages
of $A_1$, and two preimages of $\varphi^{-k_3}(A_3)$, and we cannot exclude that each of the two components of $p^{-1}(A_1)$ intersects each of the two components of $p^{-1}(\varphi^{-k_3}(A_3))$. We may get more than $m$, but no more than $m-2+4$ non-empty sets of the form
$A_1^{j_1}\cap  \varphi^{-k_2}( A_2^{j_2})\cap
  \varphi^{-k_3}(A_3^{j_3})$. More detailed, there are at most two of the connected components of $p^{-1}(A_1^{j_1})$, such that each of them may contain
$2$ different sets of the form $A_1^{j_1}\cap  \varphi^{-k_2}(A_2^{j_2})\cap\varphi^{-k_3}(A_3^{j_3})$, and all other components contain exactly one set of this form. Equivalently, there are at most two of the connected components
of  $ p^{-1}(\varphi^{-k_3}(A_3^{j_3}))$,
such that each of them may contain
$2$ sets of the form $A_1^{j_1}\cap  \varphi^{-k_2}(A_2^{j_2})\cap\varphi^{-k_3}(A_3^{j_3})$, and all other components contain exactly one set of this form.

\noindent {\bf Beginning of the formal proof.}
We make the following choices of positive numbers by induction on the natural number $\ell$.
For each $\ell$ there exists a positive number $\varepsilon_{\ell}<\varepsilon_{\ell-1}$, $ \varepsilon_1<     \frac{\sigma}{2}$,  such that for any pair of points $z_1,z_2 \in X$ the following implication holds
\begin{align}\label{eq3.1b}
d(z_1,z_2)< \varepsilon_{\ell}  \Rightarrow d(\varphi^{\pm l}(z_1) ,\varphi^{\pm l}(z_2))<\frac{\sigma}{2}\; \mbox{for}\; l=1,2,\ldots, \ell \,.
\end{align}
Further, there exists a positive number $\varepsilon'_{\ell}< \frac{\varepsilon_{\ell}}{2}$, $\varepsilon'_{\ell}<\varepsilon'_{\ell-1}$, such that for any pair of points $z_1,z_2\in X$ the implication
\begin{align}\label{eq3.1d}
d(z_1,z_2)\geq \frac{\varepsilon_{\ell}}{2} \Rightarrow d(\varphi^{\pm l}(z_1) ,\varphi^{\pm l}(z_2))\geq\varepsilon'_{\ell}\;\mbox{for}\; l=1,2,\ldots, \ell \,
\end{align}
holds.
Finally, there exists a positive number $\delta_{\ell} < \frac{\varepsilon'_{\ell}}{4}$, $\delta_{\ell}<\delta_{\ell-1}$,   such that such that for any pair of points $z_1,z_2 \in X$
\begin{align}\label{eq3.1c}
d(z_1,z_2)\leq 2\delta_{\ell} \Rightarrow d(\varphi^{\pm l}(z_1) ,\varphi^{\pm l}(z_2))<\frac{\varepsilon'_{\ell}}{2}\; \mbox{for}\; l=1,2,\ldots, \ell \, .
\end{align}
The statements follow from the uniform continuity of $ \varphi^{\pm 1}$ on the compact space $X$. The implication \eqref{eq3.1d} is equivalent to the implication
\begin{align*}
d(z_1,z_2)< \frac{\varepsilon_{\ell}}{2} \Leftarrow d(\varphi^{\pm l}(z_1) ,\varphi^{\pm l}(z_2))<\varepsilon'_{\ell}\,.
\end{align*}
Consider the cover $\mathcal{A}_{\ell}$ of $X$ whose elements are the open  balls of radius $\varepsilon_{\ell}$ in the metric $d$ centered at each point of the branch locus and the open balls of radius $\delta_{\ell}$ centered at each point of the complement of all the latter balls.

Associate to $\mathcal{A}_{\ell}$ the 
cover $\hat{ \mathcal{A}}_{\ell}'$ of $\hat X$ that consists of all connected components of $p^{-1}(A)$ for all $A \in \mathcal{A}_{\ell}$.
For $\mathcal{A}$ being one of the $\mathcal{A}_{\ell}$
the cover
$\hat{\varphi}^{-1}(\hat {\mathcal{A}}')$ consists of all sets of the following form. For a set $A \in \mathcal{A}$ we take a connected component $\hat A^j$ of $p^{-1}(A)$ and consider $\hat{\varphi}^{-1}(\hat{A}^{j})$.

The cover $\hat {\mathcal{A}}' \vee \hat{\varphi}^{-1}(\hat {\mathcal{A}}') \vee \ldots\vee   \hat{\varphi}^{-n+1}(\hat {\mathcal{A}}')$   of $\hat X$ consists of all sets of the form
\begin{equation}\label{eq3.2b}
\hat{A}_1^{j_1}\cap  \hat{\varphi}^{-1}(\hat{A}_2^{j_2})\cap\ldots\cap \hat{\varphi}^{-n+1}(\hat{A}_n^{j_n})
\end{equation}
for $A_1, \ldots, A_n \in \mathcal{A}$  and $\hat{A}_{k'}^{j_{k'}}$ being a connected component of $p^{-1}(A_{k'})$,
${k'}=1,2,\ldots,n$.
We want to compare $\mathcal{N}(\mathcal{A} \vee  \varphi^{-1}(\mathcal{A})
\vee\ldots\vee \varphi^{-n+1}(\mathcal{A}))$ and  $\mathcal{N}( \hat {\mathcal{A}}' \vee \hat{\varphi}^{-1}(\hat {\mathcal{A}}') \vee\ldots\vee \hat{\varphi}^{-n+1} (\hat {\mathcal{A}}')    )$.

\noindent {\bf 1. The case when no $A_j$ meets the branch locus.} Consider the case $n=2$.
Notice that $A_1 \cap \varphi^{-1}(A_2)\neq \emptyset$ if
the intersection $\hat{A}_1^{j_1}\cap  \hat{\varphi}^{-1}(\hat{A}_2^{j_2})$ is not empty.
Suppose $A_1$ and $A_2$ do not meet the branch locus. Then there are local inverses of $p:\hat{X}\to X$, $i_{l,A_1}:A_1\to \hat X,$ and $i_{l,\varphi^{-1}(A_2)}:\varphi^{-1}(A_2)\to \hat X$, $ l=1,\ldots,m$. Recall that $m$ is the covering multiplicity of the covering $p:\hat{X}\to X$.
We claim that the sets of the form
\begin{align}\label{eq3.2f}
\hat{A}_1^{j_1}\cap  \hat{\varphi}^{-1}(\hat{A}_2^{j_2})
\end{align}
are the images
\begin{align*}
i_{l,A_1}(A_1\cap\varphi^{-1}(A_2) )=i_{l,\varphi^{-1}(A_2)}(A_1\cap\varphi^{-1}(A_2)), \,l=1,\ldots,m\,.
\end{align*}
The claim implies that in this case the number of sets of the form \eqref{eq3.2f} equals  $m$.

The claim can be obtained as follows. We write for short $\varepsilon$ instead of $\varepsilon_{\ell}$, $\varepsilon'$ instead of $\varepsilon'_{\ell}$, and $\delta$ instead of $\delta_{\ell}$. In the considered case the center of $A_j$, $j=1,2,$ has distance in the metric $d$ at least $\varepsilon$ from the branch locus, the diameter of $A_1$ equals ${\rm diam}(A_1)=2\delta$, and by the implication \eqref{eq3.1c} the diameter of $\varphi^{-1}(A_2)$ equals ${\rm diam}(\varphi^{-1}(A_2))<\frac{\varepsilon'}{2}$. Since the sets $A_1$ and
$\varphi^{-1}(A_2)$ intersect, their union has diameter ${\rm diam}(A_1\cup \varphi^{-1}(A_2))< 2\delta +\frac{\varepsilon'}{2}<{\varepsilon'}$. Since the center of $A_1$ has distance bigger than $\varepsilon$ from the branch locus, the union $A_1\cup \varphi^{-1}(A_2)$ is
contained in an $\varepsilon'$-ball $B$ that has distance at least $\varepsilon'$ from the branch locus.
The preimage $p^{-1}(B)$ 
consists of $m$ connected components (called sheets) on which $p$ is a homeomorphism. Let $i_{l,B}$ be the inverse of the restriction of $p$ to the $l$-th sheet. Then $p^{-1}(A_{1})$ is the union of the $i_{l,B}(A_{1})$, and $p^{-1}(\varphi^{-1}(A_{2}))$  is the union of the
 $i_{l,B}(\varphi^{-1}(A_{2}))$. Hence, the non-empty sets of the form \eqref{eq3.2b} are exactly the sets $i_{l,B}(A_1 \cap \varphi^{-1}(A_2)),\, l=1,\ldots,m$, i.e. the sheets over $A_1 \cap \varphi^{-1}(A_2)$. We obtained the claim for $n=2$.

For $k>2$ we claim the following. Suppose the intersection
\begin{align}\label{eq3.1c'}
A_1 \cap \varphi^{-1}(A_2)\cap\ldots\cap  \varphi^{-k+1}(A_k)
\end{align}
is not empty and none of the $A_{k'}$ intersects the branch locus.
Then the non-empty sets of the form
\begin{equation}\label{eq3.1d'}
\hat{A}_1^{j_1}\cap  \hat{\varphi}^{-1}(\hat{A}_2^{j_2})\cap\ldots\cap \hat{\varphi}^{-k+1}(\hat{A}_k^{j_k})
\end{equation}
for $A_{k'} \in \mathcal{A}$  and $\hat{A}_{k'}^{j_{k'}}$ a connected component of $p^{-1}(A_{k'})$,
${k'}=1,2,\ldots,k,$ are the images
\begin{align}\label{eq3.1e}
i_{l,{\varphi}^{-k+1}({A}_{k})}(A_1 \cap \varphi^{-1}(A_2)\cap\ldots\cap  \varphi^{-k+1}(A_k)) \;,l=1,\ldots,m\,,
\end{align}
under the $m$ local inverses $i_{l,{\varphi}^{-k+1}({A}_{k})}$ of
$p\mid {\varphi}^{-k+1}({A}_{k})$. Moreover, in the statement we may replace the local inverses $i_{l,{\varphi}^{-k+1}({A}_{k})}$ of
$p\mid {\varphi}^{-k+1}({A}_{k})$ by the local inverses $i_{l,{\varphi}^{-k'+1}({A}_{k'})}$ of
$p\mid {\varphi}^{-k'+1}({A}_{k'})$ for any $k'=1,\ldots, k$.

The last assertion is easy to see. Indeed, if it is true for some $k'>1$, then it is true also for $k'-1$, since $A_{k'-1}\cup \varphi^{-1}(A_{k'})$ is simply connected and, hence, $\varphi^{-k'}(A_{k'-1}) \cup \varphi^{-k'+1}(A_{k'}) $ is simply connected and
on $\varphi^{-k'}(A_{k'-1}) \cap \varphi^{-k'+1}(A_{k'}) $ the equality
$i_{l,{\varphi}^{-k'+1}({A}_{k'})}= i_{l,{\varphi}^{-k'}({A}_{k'-1})}= i_{l,{\varphi}^{-k'+1}({A}_{k'})\cup {\varphi}^{-k'}({A}_{k'-1})}$ holds.

We now prove the claim by induction  for $k>2$.
Suppose it is proved for some number $k\geq 2$.
The collection of sets \eqref{eq3.1d'} computed for the number $k+1$ is the collection of intersections of the sets  \eqref{eq3.1d'} computed for the number $k$
with the sets
$\hat{\varphi}^{-k}(\hat{A}_{k+1}^{j_{k+1}})$. By the statement for $k=2$ the
collection of sets of the form $\hat{A}_k^{j_k}\cap\hat{\varphi}^{-1}(\hat{A}_{k+1}^{j_{k+1}})$
equals the collection $i_{l,A_k^{j_k}}(A_k\cap\varphi^{-1}( A_{k+1})),\,l=1,\ldots,m$.
Since $\varphi^{-k+1}$ is a homeomorphism,
the collection $\hat{\varphi}^{-k+1}(\hat{A}_k^{j_k})\cap \hat{\varphi}^{-k}(\hat{A}_{k+1}^{j_{k+1}})$ equals the collection $ i_{l,{\varphi}^{-k+1}({A}_{k})}(\varphi^{-k+1}(A_k)\cap \varphi^{-k}(A_{k+1}))$. Together with the statement for $k$ this gives the statement for $k+1$ and $k'=k$.
We proved in particular, that there are exactly $m$ non-empty sets of form \eqref{eq3.1d'} if the $A_{k'}$ do not meet the branch locus and the set  \eqref{eq3.1c'} is not empty.

\noindent {\bf 2. All $A_j$ intersect the branch locus.} Let again $n=2$.
Suppose now, that  both, $A_1$ and $\varphi^{-1}(A_2)$, intersect the branch locus. Then the intersection of each set with the branch locus consists of a single point which is the center of $A_1$ and also the preimage under $\varphi$ of the center of $A_2$.
By implication \eqref{eq3.1b} the set $A_1 \cup \varphi^{-1}(A_2)$ is contained in a disc $B_1$ of radius not exceeding $\frac{\sigma}{2}$ centered at a point in the branch locus.
Consider the connected components
${B_1}^j, j=1,\ldots,m-1,$ of  $p^{-1}(B_1)$, where  $p:{B_1}^{m-1}\to B_1$ is a double branched covering, and $p:B_1^j\to B_1$ is a homeomorphism for $j<m-1$.
Since $\hat{A}_{1}^{j_1}$ and  $\hat{\varphi}^{-1}(\hat{A}_{2}^{j_{2}})$ may intersect only if they are in the same component $B_1^j$,
there are $m-1$ non-empty sets of the form \eqref{eq3.2f} each contained in a ${B_1}^j$ and each ${B_1}^j$ contains a set of form \eqref{eq3.2f}.
In this case there are $m-1$ non-empty sets of the form \eqref{eq3.2f}.

Consider now $k$ sets $A_{k'}$ each of which intersects the branch locus. Suppose that
$A_1\cap\varphi^{-1}(A_2)\cap\ldots\cap\varphi^{-k+1}(A_k)$ is not empty.
The same induction argument as in the case when none of the $A_{k'}$ meets the branch locus
gives the following. Take a set $A_{k'}$ of the collection. The set $\varphi^{-k'+1}(A_{k'})$
contains a single point in the branch locus, namely the center of the $\varepsilon$-disc $A_1$.
The preimage $p^{-1}(\varphi^{-k'+1}(A_{k'}))$ consists of $m-1$ connected components $B_l$,
such that the restriction of $p$ to the first $m-2$ of them is a homeomorphism, and the
restriction of $p$ to the last component is a double branched covering.
For each $l=1,\ldots,m-1$ we consider the preimage of $A_1\cap\varphi^{-1}(A_2)\cap\ldots\cap\varphi^{-k+1}(A_k)$ under the projection $p$, that is contained in the connected component $B^l$ of $p^{-1}(B)$. The collection of non-empty sets of the form
\eqref{eq3.2b} is equal to this collection.

\noindent {\bf 3. The mixed case, two sets.} Suppose now that $A_1$ intersects the branch locus,  $A_2$ does not, and $A_1\cap \varphi^{-1}(A_2)\neq\emptyset$.
Then $A_1$ is a disc of radius $\varepsilon$ with center at a point in the branch locus. The ball $B_1$ of radius $2\varepsilon$ with the same center also has a single point in the branch locus.
Since $A_2$
does not intersect the branch locus, it is a disc of radius $\delta$ around a point $z$ that has distance bigger than  $\varepsilon$ to the branch locus.
By implication \eqref{eq3.1d} the preimage  $\varphi^{-1}(z)$
has distance at least $\varepsilon'$ from the branch locus.
By \eqref{eq3.1c} the diameter of $\varphi^{-1}(A_2)$
is less than $\frac{\varepsilon'}{2}$.
Hence, the distance of  $\varphi^{-1}(A_2)$ from the branch locus is not smaller than $\varepsilon'-\frac{1}{2}\varepsilon'=\frac{1}{2}\varepsilon'$. Hence,  $\varphi^{-1}(A_2)$ is contained in a sector $S_1$ of an angle not bigger than $\pi$ of the disc $B_1$. The preimages of $S_1$ under $p$ are $m$ sheets that are mapped by $p$ homeomorphically onto $S_1$.
It is now clear that the collection of sets  \eqref{eq3.2f} is equal to the collection of the $m$
images of the local inverses of $p$ on $S_1$.

\noindent {\bf 4. The building block for the general case.} $\,$ Consider now elements $\;$ $A_{k_0}, A_{k_0+1},\ldots,A_{k_0+k_1},A_{k_0+k_1+1}$ of the cover $\mathcal{A}_{\ell}$ such that $A_{k_0}$ and $A_{k_0+k_1+1}$ do not intersect the branch locus, but all other sets are discs around a point in the  branch locus. Suppose
${\varphi}^{-k_0+1}({A}_{k_0})\cap\ldots\cap {\varphi}^{-k_0-k_1
}({A}_{k_0+k_1+1})\neq \emptyset$ and $k_1\leq \ell-1$. We claim that the collection of sets of the form
\begin{align}\label{eq3.1f}
 \hat{\varphi}^{-k_0+1}(\hat{A}_{k_0}^{j_{k_0}})\cap\ldots\cap \hat{\varphi}^{-k_0-k_1
}(\hat{A}_{k_0+k_1+1}^{j_{k_0+k_1+1}})
\end{align}
equals
the collection of the $m$ images of the set
$$
\varphi^{-k_0+1}(A_{k_0})\cap \ldots\cap    {\varphi}^{-k_0-k_1
}({A}_{k_0+k_1+1})
$$
under the local inverses of $p$ on the set ${\varphi}^{-k_0-k_1
}({A}_{k_0+k_1+1})$.

By the preceding arguments
the respective statement is true for the sets of this collection that do intersect the branch locus. Namely, the collection of sets of the form
\begin{align}\label{eq3.1g}
 \hat{\varphi}^{-k_0}(\hat{A}_{k_0+1}^{j_{k_0+1}})\cap\ldots\cap \hat{\varphi}^{-k_0-k_1
+1}(\hat{A}_{k_0+k_1}^{j_{k_0+k_1}})
\end{align}
equals the collection of preimages of ${\varphi}^{-k_0}({A}_{k_0+1})\cap\ldots\cap {\varphi}^{-k_0-k_1+1}({A}_{k_0+k_1})$ contained in the $m-1$ connected components of the preimage $p^{-1}({\varphi}^{-k'+1}({A}_{k'}))$ for any $k'$ between $k_0+1$ and $k_0+k_1$.

Recall that the two sets $A_{k_0}$ and $A_{k_0+k_1+1}$, that do not intersect the branch locus, are discs of radius $\delta_{\ell}$ with center at distance more than $\varepsilon_{\ell}$ from the branch locus. Also,
 $A_{k_0}\cap \varphi^{-k_1-1}(A_{k_0+k_1+1})\neq \emptyset$. By the implication 
\eqref{eq3.1c} the diameter of $ \varphi^{-k_1-1}(A_{k_0+k_1+1})$ is smaller then $\frac{1}{2}\varepsilon'_{\ell}$. Hence, $A_{k_0}\cup \varphi^{-k_1-1}(A_{k_0+k_1+1})$ is contained in a disc $B'_{k_0}$ of radius not exceeding $\delta_{\ell}+\frac{1}{2}\varepsilon'_{\ell}\leq \varepsilon'_{\ell}$ around the center of  $A_{k_0}$. 
Since the distance of the center of $A_{k_0}$ to the branch locus is at least $\varepsilon_{\ell}$, the distance of $B'_{k_0}\supset A_{k_0}\cup \varphi^{-k_1-1}(A_{k_0+k_1+1})$ to the branch locus is at least $\varepsilon_{\ell}-{\varepsilon'_{\ell}}\geq\frac{\varepsilon_{\ell}}{2}\geq \varepsilon'_{\ell} $.

By \eqref{eq3.1b}  $\varphi ^{-1}( A_{k_0+1})$ is contained in the disc of radius $\frac{\sigma}{2}$ around the point in the branch locus that is contained in $\varphi^{-1} ( A_{k_0+1})$. Hence, the union
${A}_{k_0}\cup \varphi^{-k_1-1}({A}_{k_0+k_1+1})\cup\varphi^{-1} ( {A}_{k_0+1})$ is contained in the $\sigma$-disc around a point of the branch locus. It is now clear, that the collection of sets
$\hat{A}_{k_0}^{j_{k_0}}\cap \hat{\varphi}^{-k_1-1}(\hat{A}_{k_0+k_1+1}^{j_{k_0+k_1+1}})\cap\hat{\varphi}^{-1} ( \hat{A}_{k_0+1}^{j_{k_0+1}})$ is equal to the collection of the $m$ images of the set $A_{k_0}\cap \varphi^{-k_1-1}(A_{k_0+k_1+1})\cap\varphi^{-1} ( A_{k_0+1})$ under the
local inverses of $p$ on the disc $B'_{k_0}$ of radius $\varepsilon'_{\ell}$ that contains $A_{k_0}\cup \varphi^{-k_1-1}(A_{k_0+k_1+1})$.
Apply the diffeomorphism $ \varphi^{-k_0+1}$. We see that the collection of the non-empty sets of the form
$\hat{\varphi}^{-k_0+1}(\hat{A}_{k_0}^{j_{k_0}})\cap     \hat{\varphi}^{-k_0}(\hat{A}_{k_0+1}^{j_{k_0+1}})\cap  \hat{\varphi}^{-k_0-k_1}(\hat{A}_{k_0+k_1+1}^{j_{k_0+k_1+1}})$ is the collection of the $m$ images of the set $\varphi^{-k_0+1}(A_{k_0})\cap\varphi^{-k_0} ( A_{k_0+1})\cap
 \varphi^{-k_0-k_1}(A_{k_0+k_1+1})$ under the local inverses of $p$ on the set $\varphi^{-k_0+1}(A_{k_0})$ (equivalently,
under the local inverses of $p$ on the set $\varphi^{-k_0-k_1}(A_{k_0+k_1+1})$). Hence the same is true for the collection of sets
 \eqref{eq3.1f}. The claim is proved.

\noindent {\bf 5. The general case.} Take now a maximal collection of consecutive sets $\;$ $A_{k_0+1},\ldots,A_{k_0+k_1}$, that intersect the branch locus. We assume that $k_0+1>1$ and $k_0+k_1<n$  and
$k_1\geq \ell$. Notice that $A_{k_0}$ and  $A_{k_0+k_1+1}$ do not meet the branch locus.
By the remark in the beginning of the proof there are no more than $m-2+4<2m$ different non-empty sets of the form $\hat{\varphi}^{-k_0+1}(\hat{A}_{k_0}^{j_{k_0}})\cap\ldots\cap  \hat{\varphi}^{-k_0-k_1}(\hat{A}_{k_0+k_1+1}^{j_{k_0+k_1+1}})$, and each set is contained in a connected component of $p^{-1}(  {\varphi}^{-k_0-k_1}({A}_{k_0+k_1+1}^{j_{k_0+k_1+1}})   )$ for the set ${A}_{k_0+k_1+1}$ that does not intersect the branch locus. Moreover, each connected component of $p^{-1}(\hat{\varphi}^{-k_0-k_1}({A}_{k_0+k_1+1}))$
 contains at most two sets of the form $\hat{\varphi}^{-k_0+1}(\hat{A}_{k_0}^{j_{k_0}})\cap\ldots\cap  \hat{\varphi}^{-k_0-k_1}(\hat{A}_{k_0+k_1+1}^{j_{k_0+k_1+1}})$.

Consider now
any collection $A_1,A_2,\ldots,A_n\in \mathcal{A}_{\ell}$. Divide it into maximal collections of consecutive sets, so that either all sets of a collection intersect the branch locus or all sets of the collection do not intersect the branch locus.
For each integer number $k$ we consider all the maximal collections of
consecutive sets $A_{k_0+1},\ldots, {A}_{k_0+k_1}$, that intersect the branch locus and for which  $k_1\geq \ell$ and $k_0+k_1\leq k$. Its number is denoted by $\mathfrak{n}(k)$.
By induction on the number of the collection we obtain the following.

If $A_{k_0+1},\ldots, {A}_{k_0+k_1}$ is a maximal collection of consecutive sets among the $A_1,A_2,\ldots,A_n\in \mathcal{A}_{\ell}$,    that do not intersect the branch locus, then each connected component of the preimage $p^{-1}({\varphi}^{-k_0-k_1-1}({A}_{k_0+k_1}^{j_{k_0+k_1}}))$ contains at most $2^{\mathfrak{n}(k_0 )}=2^{\mathfrak{n}(k_0 +k_1)}$ different sets of the form $\hat{A}_1^{j_1}\cap\ldots\cap  \hat{\varphi}^{-k_0-k_1-1}(\hat{A}_{k_0+k_1}^{j_{k_0+k_1}})$.
If $A_{k_0+1},\ldots, {A}_{k_0+k_1}$ is a maximal collection of consecutive sets that  intersect the branch locus and  $k_0+k_1\leq n$, then,  if   $k_1< \ell$, each connected component of
$p^{-1}(  {\varphi}^{-k_0-k_1+1}({A}_{k_0+k_1}^{j_{k_0+k_1}})   )$
contains at most  $2^{\mathfrak{n}(k_0 )}$ different sets of the form $\hat{A}_1^{j_1}\cap\ldots\cap  \hat{\varphi}^{-k_0-k_1-1}(\hat{A}_{k_0+k_1}^{j_{k_0+k_1}})$,
and contains at most
$2\cdot 2^{\mathfrak{n}(k_0  )}= 2^{\mathfrak{n}(k_0+k_1  )}$ different sets of the form $\hat{A}_1^{j_1}\cap\ldots\cap  \hat{\varphi}^{-k_0-k_1+1}(\hat{A}_{k_0+k_1}^{j_{k_0+k_1}})$, if $k_1\geq\ell$.

Since $\mathfrak{n}(n)\leq [\frac{n}{\ell}]$, we obtain
$$\mathcal{N}\Big(\hat{\mathcal{A}}_{\ell}\vee\ldots \vee\hat{\varphi}^{-n}(\hat{\mathcal{A}}_{\ell})\Big)
\leq 2^{ [\frac{n}{\ell}]}\cdot m \cdot\mathcal{N}\Big(\mathcal{A}_{\ell}\vee\ldots \vee\varphi^{-n}(\mathcal{A}_{\ell})\Big)\,.$$
Hence,
\begin{equation}\label{eq3.2e}
h(\hat{\mathcal{A}}_{\ell},\hat\varphi)\leq \frac{\log 2}{\ell}+ h(\mathcal{A}_{\ell},\varphi)\,
\end{equation}
for each of the described covers $\mathcal{A_{\ell}}$ of $X$.
Since the sequences $\mathcal{A}_{\ell}$ and $\hat{\mathcal{A}}_{\ell}$ are refining we obtain
\begin{equation}\label{eq3.7}
h(\hat\varphi) \leq h(\varphi) \, .
\end{equation}
The lemma is proved. \hfill $\Box$

\bigskip

Theorem \ref{thm3.1} and Lemma \ref{lem3.4} together with Ahlfors' trick (see Section \ref{sec:2.3a})
will  imply the following.

\begin{thm}\label{thm3.1'} Let $X$ be a closed connected Riemann surface
of genus $g$ with a set $E_m$ of $m\geq 0$ distinguished points, $3g-3+m>
0$. Then any non-periodic absolutely extremal self-homeomorphism
of $X$ with set of distinguished
points $E_m$ is entropy minimizing in its
isotopy class.
\end{thm}

Together with Corollary \ref{corr2.1} we obtain the following
statement, which includes homeomorphisms with elliptic or hyperbolic
modular transformation.

\begin{cor}\label{corr3.3} Let $\varphi$ be an irreducible
self-homeomorphism of a closed connected Riemann surface $X$ with
a set $E_m$ of $m\geq 0$ distinguished points, $3g-3+m>0$. Let $\varphi_0$
be an absolutely extremal self-homeomorphism of a Riemann surface
$Y$ with $m$ distinguished points, which is obtained from $\varphi$
by isotopy and conjugation. Then

\begin{eqnarray}
h(\varphi_0) &= &\frac12 \log K(\varphi_0 ) =  L (\varphi^*)  =  h(\widehat {{\mathfrak m}_{\varphi _0}} ) \nonumber \\
&= &{\rm inf} \, \{h(\varphi_1) : \varphi_1 \ \mbox{is obtained from
$\varphi$ by isotopy and conjugation}\} \, . \nonumber
\end{eqnarray}
\end{cor}

\noindent {\bf Proof of Theorem \ref{thm3.1'}.}
Let $X$ be a
connected closed Riemann surface with set of distinguished points
$\{ z_1 , \ldots , z_m\}$, $3g-3+m>0$. According to \cite{A2}
(see also  Section \ref{sec:2.3a})
there
is a closed Riemann surface $\hat X$ of genus at least two which is
a simple branched covering of $X$ so that the set of branch points
projects onto the set $\{ z_1 , \ldots , z_m\}$. The non-periodic absolutely extremal self-homeomorphism
$\varphi_0$ lifts to a non-periodic absolutely extremal self-homeomorphism
$\hat\varphi_0$ on $\hat X$ with $K(\varphi_0) = K(\hat\varphi_0)$.

By the lemma
\begin{eqnarray}
&&{\rm inf} \, \{h(\varphi_1) : \varphi_1 \ \mbox{is a
self-homeomorphism of $X$
which is isotopic to $\varphi$}\} \nonumber \\
& = &{\rm inf} \, \{h(\hat\varphi_1) : \hat\varphi_1 \ \mbox{is the
lift of a
self-homeomorphism of $X$} \nonumber \\
& &\mbox{which is isotopic to $\varphi$}\} \nonumber \\
& \geq & {\rm inf} \, \{h({\mathcal F}_1) : {\mathcal F}_1 \
\mbox{is a self-homeomorphism of $\hat X$ which is isotopic to
$\hat\varphi$}\}. \nonumber
\end{eqnarray}

Since $\hat\varphi_0$ is absolutely extremal on $\hat X$ and isotopic to
$\hat\varphi$, the Fathi-Shub Theorem \ref{thm3.1} implies that the
last infimum equals $h(\hat\varphi_0)$. Thus, the last infimum is
attained on a homeomorphism of $\hat X$ which is a lift. Hence the
inequality between the second and the third infimum is an equality
too. Since by Lemma \ref{lem3.4} we have $h(\varphi_0) =
h(\hat\varphi_0)$, the first infimum equals $h(\varphi_0)$. Also,
$K(\varphi_0) = K(\hat\varphi_0)$. It remains to use
Corollary \ref{corr2.1}, and the invariance of the
entropy under conjugation. \hfill $\Box$

\section[The trajectories of quadratic differentials]{Absolutely extremal homeomorphisms and the trajectories of the associated quadratic differentials. }
\label{sec:entropy.2a}

We collect here some facts concerning trajectories of quadratic differentials that are related to non-periodic absolutely extremal self-homeomorphisms.
For a more comprehensive exposition we refer the reader to the book of Strebel \cite{Str}.

Let $\phi$ be a quadratic differential associated to a non-periodic absolutely extremal self-homeomorphism $\varphi$ on a closed Riemann surface $X$ possibly with set of distinguished points $E$. Recall that $\phi$ is a meromorphic quadratic differential. It is analytic except, maybe, at some distinguished points where it has at worst first order poles. We need some information about the horizontal and vertical trajectories of $\phi$.

For a horizontal trajectory $\gamma$ and a point $p$ on it we call the union of $\{p\}$ with any of the connected components of $\gamma \setminus \{p\}$ a horizontal trajectory ray. 
If the set of limit points of a trajectory ray (that is the closure of the ray in the $\phi$-metric minus the ray itself) contains a regular point it contains a horizontal arc through this point. If the limit set of a trajectory ray
consists of a single point this point must be critical and the trajectory ray is called critical. \index{trajectory ! ray} \index{trajectory ! critical trajectory ray}

\begin{prop}\label{prop3.101}
For a quadratic differential $\phi$ associated to a pseuo-Anosov self-homeomorphism $\varphi$ of a closed Riemann surface (possibly with distinguished points)
there is no leaf of the horizontal foliation with both trajectory rays critical (and no leaf of the vertical foliation with this property).
\end{prop}
\noindent {\bf Proof.} Suppose there is a horizontal leaf joining two critical points. There are only finitely many critical points and in a neighbourhood of each critical point there are only finitely many horizontal bisectrices. 
Since
$\varphi$ maps the set of singular points of $\phi$ onto itself (perhaps, permuting the points) and $\varphi$ maps horizontal trajectories to itself,
a power $\varphi^n$ of $\varphi$ maps the leaf onto itself. Adding the two endpoints to the leaf we obtain a compact curve which has finite $\phi$-length. But by
\eqref{eq2.17a'} the homeomorphism  $\varphi^n$ expands the length of the horizontal leaves by the factor $K^{\frac{n}{2}}$ which is impossible. \hfill $\Box$

\begin{prop}\label{prop3.102}
There is no closed leaf of the horizontal foliation and no closed leaf of the vertical foliation of a quadratic differential $\phi$ associated to a pseuo-Anosov homeomorphism $\varphi$.
\end{prop}
\noindent {\bf Proof.} Suppose there is a closed horizontal leaf $\gamma$.
Choose an orientation for $\gamma$.
We claim that there is an annulus $A$ around this leaf which is contained in the regular part of $X$ and is the union of closed horizontal leaves. Indeed, each point $x$
of the leaf $\gamma$ is regular. Hence, there is a neighbourhood $V$ of each point $x$, that is
equipped with distinguished coordinates centered at the point, such that
$V\cap \gamma$ is the set of points in $V$ that are real in distinguished coordinates, 
and the direction of the positive orientation of $\gamma$ corresponds to the positive direction of the real axis. We may assume that, perhaps after shrinking, the neighbourhood of each $x$ is 
a square $R(x)$  with center $x$ written in the distinguished coordinates centered at $x$ as
$\{|\mbox{Re}\,z)|< \varepsilon(x), \, |\mbox{Im}\,z|< \varepsilon(x)\}$. 
Cover $\gamma$ by a finite number of such squares labeled by $R_j, j=1,\ldots,N,$ of side length $\varepsilon_j$. We may do this so, that considering an infinite sequence of squares with $R_{j+N}=R_j$ for each $j$, any $R_j$ intersects
exactly the squares $R_{j\pm1}$. Let $\varepsilon$ be the minimum of the $\varepsilon_j$.

The squares $R_j$ cover an annulus $A$ around $\gamma$. By the choice of the distinguished coordinates on the $R_j$ each transition function is a translation by a real number. The arcs given in distinguished coordinates on $R_j$ by $(-\varepsilon_j, \varepsilon_j)$ cover the closed horizontal curve $\gamma$. Hence, for any real parameter $t$, $ |t| < \epsilon \stackrel{def}= \min \varepsilon_j$, the arcs $(-\varepsilon_j + i t, \varepsilon_j + i t)$ cover a closed horizontal curve. The union of these curves is an annulus $A$ around $\gamma$ whose existence was claimed.

Take any boundary point of the annulus $A$ in the sense, that the point is a limit point of $A$ 
in the $\phi$-metric and is not contained in $A$. If this point is a regular point then there is a horizontal arc contained in $\partial A$ 
that contains the point. 
If all points of a boundary component of $A$ are regular, then by the preceding arguments there is a bigger annulus that contains the previous annulus as well as the boundary component, and is the union of closed horizontal curves. Take the maximal annulus containing $\gamma$ that is the  union of closed horizontal curves. Then each boundary component must contain a critical point and the complement in each boundary component of the critical points consists of horizontal leaves with both rays being critical. This is not possible by Proposition \ref{prop3.101}. \hfill $\Box$

\medskip

We call a trajectory ray divergent if its limit set contains more than one point. A trajectory ray with initial point $p$ is called recurrent if $p$ is contained in the limit set of the ray. We will consider trajectory rays parameterized so that the initial point has the smallest parameter and write $p_2>p_1$ for two points $p_1$ and $p_2$ on the ray if the parameter of $p_1$ is smaller than that of $p_2$. The parametrization defines an orientation of the ray. \index{trajectory ! divergent trajectory ray}

Let as before $X$ be a Riemann surface with quadratic differential $\phi$, and let $R$ be an open rectangle (and let $\bar R$ be a closed rectangle, respectively)
in the complex plane with sides parallel to the axes . A mapping
$\mathcal{F}:R\to X$
is called an open $\phi$-rectangle \index{$\phi$-rectangle}
(and a mapping $\mathcal{F}:\bar R\to X$ is called a closed $\phi$-rectangle, respectively,) if  $\mathcal{F}$ takes horizontal segments in $R$ (in $\bar R$, respectively) to $\phi$-horizontal arcs and takes vertical segments in $R$ (in $\bar R$, respectively) to $\phi$-vertical arcs, so that the Euclidean length of the horizontal (vertical arc, respectively) is equal to the $\phi$-length of its image.

Recall that trajectories do not contain $\phi$-singular points, hence the images of $\phi$-rectangles are contained in the $\phi$-regular part of $X$.
We call a $\phi$-rectangle embedded if the mapping $\mathcal{F}$ is injective.

\begin{thm}\label{thm3.03}
Every divergent trajectory ray is recurrent. More precisely, let $\alpha^+$ be a divergent horizontal trajectory ray with initial point $p_0$. 
Let $\beta$ be any oriented open arc through $p_0$ contained in a vertical leaf, that is cut positively by $\alpha^+$ at $p_0$. 
Then for each point $p_1$ on the ray $\alpha^+$ there exists a point $p_2> p_1$ on $\alpha^+$, such that $\alpha^+$  cuts $\beta$ positively at $p_2$. The $\phi$-length of the arc on $\alpha^+$ between $p_1$ and $p_2$ is bounded from below by a constant $c$ that depends only on the quadratic differential $\phi$.
\end{thm}
\index{trajectory ! divergent}
\noindent {\bf Proof.} Let $\beta'\subset \beta$ be a vertical arc which has center $p_0$ and $\phi$-length $\ell_{\phi}(\beta')$, such that the vertical arc with center $p_0$ and $\phi$-length $3\ell_{\phi}(\beta')$ is contained in $\beta$. 
We call a point $p\in \beta'$ exceptional, if the horizontal trajectory ray, which cuts $\beta'$ positively at $p$, is critical and does not cut $\beta'$ positively at a point greater than $p$. Notice that a point $p$ is not exceptional, if the trajectory ray which cuts $\beta'$ positively at $p$ is critical but cuts $\beta'$ positively at a point greater than $p$ before running into the critical point.
We claim that only finitely many points of $\beta'$ are exceptional.
Indeed,
each such point $p$ and the respective critical ray are obtained as follows. Take a critical point, take any
oriented trajectory emerging from this critical point, consider its part between the critical point and its first negative intersection point with $\beta'$ (if there is such), and invert orientation.
There are only finitely many such rays, hence the claim.

Recall that the ray $\alpha^+$ is not critical. We take a closed arc $\beta_0$, that is contained in $\beta'$, has midpoint $p_0$,  and does not contain exceptional points.
We will prove now the following\\
{\bf Claim} {\it There is a horizontal trajectory ray that emerges at a point $p\in \beta_0$ and intersects $\beta'$ positively.}

For positive numbers $t$ we consider the closed rectangle $\overline{R_t} \subset \mathbb{C}$ with sides parallel to the axes, with horizontal Euclidean side length equal to $t$, 
vertical Euclidean side length
$\ell_{\phi}(\beta_0)$, and
with midpoint of the left side equal to the origin.

For small $t$ there is a closed $\phi$-rectangle $\mathcal{F}_{\phi,t}: \overline{R_t}\to X$ which maps the left side of $\overline{R_t}$ onto $\beta_0$.
Such a $\phi$-rectangle can be defined for parameters $t$ as long as each horizontal trajectory ray that emerges at a point $p\in\beta_0$ in positive direction does not run into a critical point 
before running along a closed arc of $\phi$-length $t$ with initial point $p$.

Assume by contradiction that the claim is not true. 
Since $\beta_0$ does not contain exceptional points,
each trajectory ray that emerges at a point in $\beta_0$ in positive direction is non-critical (otherwise it would intersect $\beta'$ positively at a point after its initial point). Hence, a closed $\phi$-rectangle $\mathcal{F}_{\phi,t}: \overline{R_t}\to X$ which maps the left side of $\overline{R_t}$ onto $\beta_0$ exists for all $t>0$.

\begin{figure}[H]
\begin{center}
\includegraphics[width=85mm]{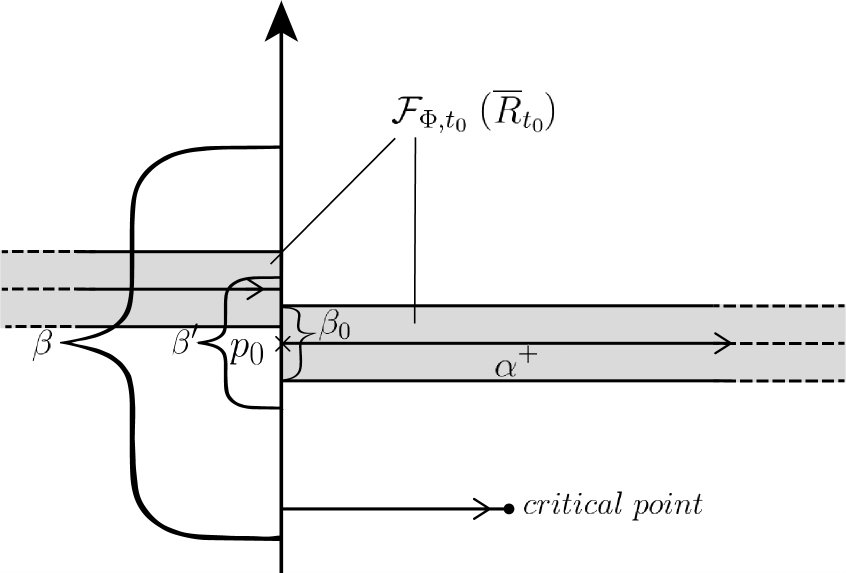}
\end{center}
\caption{The maximal injective $\phi$-rectangle $\mathcal{F}_{\phi,t_0}(\overline{R_{t_0}})$}\label{fig3.1}
\end{figure}

If $\mathcal{F}_{\phi,t}$ is injective on $\overline{R_t}$, the image $\mathcal{F}_{\phi,t}(\overline{R_t})$ 
has $\phi$-area $t \cdot \ell_{\phi}(\beta_0)$.
Since the $\phi$-area $\|\phi\|_1$ of $X$ is finite, there exists a smallest number $t_0>0$ for which the $\phi$-rectangle $\mathcal{F}_{\phi,{t_0}}:\overline{R_{t_0}}\to X$ is not injective.
Denote by  ${\sf s}_{t_0}$ the right side of $\overline{R_{t_0}}$. 
The mapping $\mathcal{F}_{\phi,t_0}$ is injective on $\overline{R_{t_0}}\setminus {\sf s}_{t_0}$. Hence, $\mathcal{F}_{\phi,t_0}( {\sf s}_{t_0})$ intersects the rest of the boundary of $\mathcal{F}_{\phi,t_0}(\overline{R_{t_0}})$.
We claim, that the vertical segment $\mathcal{F}_{\phi,t_0}( {\sf s}_{t_0})$ cannot intersect the image of the open horizontal sides of $\overline{R_{t_0}}$.
Suppose, in contrary,  $\mathcal{F}_{\phi,t_0}( {\sf s}_{t_0})$ intersects the image of an open horizontal side ${\sf s}_h$ of $\overline{R_{t_0}}$. Let $p'$ be the intersection point. 
Then $p'=\mathcal{F}_{\phi,t_0}(z')=\mathcal{F}_{\phi,t_0}(z'')$, where
$z'$ is on an open horizontal side ${\sf s}_h$ of $\overline{R_{t_0}}$, and $z''=t_0+iy$ is contained in ${\sf s}_{t_0}$. But then for all $t<t_0$ and close to $t_0$ the point $\mathcal{F}_{\phi,t_0}(t+iy)$ is also contained in 
the image $\mathcal{F}_{\phi,t_0}({\sf s}_h)$ of the open horizontal side ${\sf s}_h$. 
But this means,
that $\mathcal{F}_{\phi,t_0}$ is not injective on $\overline{R_{t}}$ for $t<t_0$ and close to $t$, contrary to the definition of $t_0$.

Hence,
$\mathcal{F}_{\phi,t_0}( {\sf s}_{t_0})$ intersects the image of the left side of $\overline{R_{t_0}}$, i.e $\mathcal{F}_{\phi,t_0}( {\sf s}_{t_0})$ intersects $\beta_0\subset \beta'$. Since $\mathcal{F}_{\phi,t_0}$ is injective on $\overline{R_{t_0}}\setminus {\sf s}_{t_0}$,
the points on $\overline{R_{t_0}}$ that are close to different
vertical sides of the rectangle, but not on the vertical sides, are mapped to different (local) sides of the vertical arc containing $\beta_0$. 
This means, that there is a horizontal trajectory that starts at a point $\beta_0$ and cuts $\beta_0\subset \beta'$ positively after running along a segment of length $t_0$. 
This contradicts our assumption that the claim is not true. The claim is proved.

Let $t_0$ be the infimum of positive numbers $t$ for which a horizontal trajectory ray that emerges at a point of $\beta_0$ in positive direction intersects $\beta'$ positively after running along a segment of length $t$.
Then, since $\beta_0$ does not contain exceptional points, the closed $\phi$-rectangle $\mathcal{F}_{\phi,t_0}: \overline{R_{t_0}}\to X$ exists. Moreover, the image of the right side
$\mathcal{F}_{\phi,t_0}( {\sf s}_{t_0})$ of $\overline{R_{t_0}}$ intersects $\beta'$. 
Then $\mathcal{F}_{\phi,t_0}( {\sf s}_{t_0})$
is a vertical arc of $\phi$-length $\ell_{\phi}(\beta_0)$ that intersects $\beta'$, hence, it is contained in $\beta$.
This means that each trajectory ray that emerges at a point of $\beta_0$ in positive direction, in particular $\alpha_+$, intersects $\beta$ positively after running along a segment of length $t_0$.

We proved that $\alpha_+$ cuts $\beta$ positively after $p_0$.
Denote the intersection point by $p_1$. The same argument applied to $p_1$ instead of $p_0$ shows that $\alpha_+$ cuts $\beta$ positively at a point $p_2$ after $p_1$.

There is a positive constant $c$ such that the $\phi$-length of an arc on $\alpha^+$ between two distinct positive intersection points with $\beta$ exceeds $c$. Indeed, such an arc must leave a simply connected neighbourhood of the closure of $\beta$, that depends only on the quadratic differential $\phi$.
Repeating the argument for the already found intersection points with $\beta$ we obtain
a sequence of positive intersection points $p_0<p_1<\ldots <p_n< \ldots $   of $\alpha_+$ with $\beta$ such that the $\phi$-length of the arc on $\alpha^+$ between $p_0$ and $p_n$ tends to $\infty$.
The theorem is proved.
\hfill $\Box$

\begin{figure}[H]
\begin{center}
\includegraphics[width=75mm]{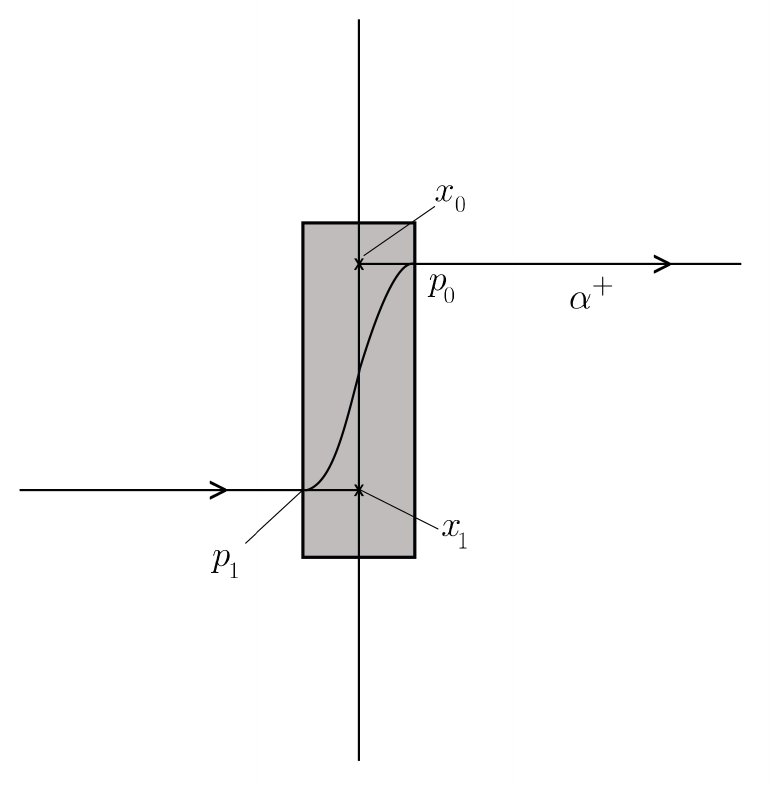}
\end{center}
\caption{A recurrent horizontal trajectory gives to a simple closed curve that is transversal to the horizontal foliation}\label{fig3.4}
\end{figure}

\begin{thm}\label{thm3.04}
For each $\phi$-regular point $x_0 \in X$ there exists a simple closed smooth curve in $X$ that passes through $x_0$, is contained in the $\phi$-regular part of $X$, and is transversal to the vertical foliation. The horizontal $\phi$-length $\ell_{\phi,h}$ of this curve is positive.
\end{thm}

\noindent {\bf Proof.}
By Proposition \ref{prop3.101} one of the horizontal trajectory rays that start at $x_0$ is not critical and therefore recurrent. We consider this recurrent trajectory ray, denoted by $\alpha^+$, oriented in the direction away from $x_0$. Let $\beta$ be a small vertical arc through $x_0$ that is oriented so that $\alpha^+$ intersects it positively at $x_0$. 
By Theorem \ref{thm3.03} $\alpha^+$ intersects $\beta$ positively at a point $x_1$ after $x_0$. 
Surround the closed vertical arc $\beta^*\subset\beta$ that joins the points $x_0$ and $x_1$ 
by an open $\Phi$-rectangle $V$ with small horizontal side length. Let $p_0$ be the first point (after $x_0$) of intersection of $\alpha_+$ with $\partial V$, and $p_1$ the last (before $x_1$) intersection point of $\alpha_+$ with $\partial V$. We take the closed curve which is the union of the part of $\alpha_+$ between $p_0$ and $p_1$ and a curve in $V$ that joins $p_1$ with $p_0$ and is transversal to the vertical foliation (see Figure \ref{fig3.4}). The curve can be chosen to be smooth and simple closed.
Its horizontal length equals the length of the part of $\alpha_+$ between $x_0$ and $x_1$.

The theorem is proved. \hfill $\Box$

\section[Pseudo-Anosov self-homeomorphisms with distinguished points]{The entropy of Pseudo-Anosov self-homeomorphisms of surfaces with distinguished points}
\label{sec:entropy.3}
In this section we will prove Theorem \ref{thm3.2}. Our proof will be different from that in \cite{FLP}. We begin with an especially simple example which is the entropy of a pseudo-Anosov self-homeomorphism of a torus with a distinguished point. This example is considered in \cite{AKM}. The arguments given here will be adapted later to give a proof of the general case of the theorem.

Consider the standard torus $X=\mathbb{C}\diagup ({\mathbb{Z}+ i\mathbb{Z}})$ with distinguished point $0 \diagup ({\mathbb{Z}+ i\mathbb{Z}})$. Let $\varphi$ be a self-homeomorphism of $X$ that lifts to a real linear self-homeomorphism of the complex plane which maps the integer lattice onto itself. Then $\varphi$ fixes the distinguished point $0 \diagup ({\mathbb{Z}+ i\mathbb{Z}})$ of $X$. The lift $\tilde \varphi$ of $\varphi$ to the universal covering $\mathbb{C}$  corresponds to a matrix $\begin{pmatrix}a&b\\c&d\end{pmatrix}$ with integer entries and determinant equal to one, i.e.
$$
\tilde\varphi \begin{pmatrix} \xi \\ \eta \end{pmatrix} \,=\, \begin{pmatrix} a&b\\c&d \end{pmatrix}  \begin{pmatrix} \xi \\ \eta \end{pmatrix}
$$
where $\zeta=\xi + i \eta \in \mathbb{C}$ are complex coordinates on $\tilde X =\mathbb{C}$, and $ad-bc=1$.
Suppose the eigenvalues of the matrix are positive and not equal to each other. Denote the bigger eigenvalue by $\lambda$ and the corresponding eigenvector by $v_+$. The other eigenvalue equals $\frac{1}{\lambda}$. Denote the corresponding eigenvector by $v_-$. Then $\tilde\varphi(x v_+ + y v_-)=\lambda x v_+ +  \lambda^{-1} y v_-$.
Consider the conformal structure $w:X\to w(X)$ on $X$ whose lift
$\tilde{w}: \tilde{X}  \to \tilde{w}(\tilde X)$ maps the point $x v_+ + y v_-$ to the point $z=x+iy$. Then
\begin{equation}\label{eq3.130}
\tilde{w}\circ\tilde{\varphi}\circ\tilde{w}^{-1}(x  + iy )=\lambda x +  i \lambda^{-1}  y \,.
\end{equation}

Consider the self-homeomorphism $\varphi_0\stackrel{def}=w \circ \varphi \circ w^{-1}$ of $w(X)$  with distinguished point $0 \diagup \tilde{w}(\mathbb{Z}+i \mathbb{Z})$. Take local flat coordinates $z$ vanishing at a point $z_0$ and local flat coordinates $\zeta$ vanishing at the point $\varphi_0(z_0)$.
The mapping $\varphi_0$ in these local coordinates becomes
\begin{equation}\label{eq3.131}
\zeta= w \circ \varphi \circ w^{-1}(x+iy)=\lambda x + i \frac{1}{\lambda} y\,.
\end{equation}
The Beltrami differential of the mapping is
\begin{equation}\label{eq3.131a}
\frac{\overline{\partial}\varphi_0(z)}{{\partial}\varphi_0(z)} = \frac{(\lambda- \lambda^{-1})\;d\bar{z}}{(\lambda+ \lambda^{-1})\; dz}=\frac{(\lambda- \lambda^{-1})\;d\bar{z}\,dz}{(\lambda + \lambda^{-1}) \;(dz)^2}\,.
\end{equation}
The equation shows the following. Let the quadratic differential $\phi$ be given in local coordinates by $(dz)^2$. The Beltrami differential of the mapping
$\varphi_0$ equals $k \frac{|\phi|}{\phi}=k \frac{d\bar{z}}{d{z}}$ with $k=\frac{(\lambda- \lambda^{-1})}{(\lambda+ \lambda^{-1})}$. Hence, the quadratic differential
of $\varphi_0$ is the globally defined quadratic differential $(dz)^2$ on $w(X)$, and $\varphi_0$ has quasiconformal dilatation $K={\frac{1+k}{1-k}}= \lambda^2$.

The quadratic differential of the inverse mapping $\xi+i \eta \to \frac{\xi}{\lambda} + i \lambda \eta$ has Beltrami differential
$$
\frac{(\lambda^{-1}-\lambda)d\bar z \, dz}{(\lambda^{-1}+\lambda)(dz)^2}\,.
$$
Hence, the quadratic differential of the inverse mapping is $-(dz)^2$, and therefore the terminal quadratic differential of $\varphi_0$ is equal to its initial quadratic differential $\phi$. We proved that $\varphi_0$ is absolutely extremal on $w(X)$ with distinguished point and its quadratic differential equals $\phi=(dz)^2$.
Notice that the quadratic differential $\phi$
does not have singular points, the distinguished point is also a regular point. 
Such a situation is exceptional and especially easy to treat.

The following proposition on the entropy of the mapping holds.
\begin{prop}\label{prop3.103 }
$h(\varphi)= h(\varphi_0)= \log \lambda\,.$
\end{prop}
\noindent {\bf Proof.} Choose a small positive number $\varepsilon_0$ so that the projection $\mathbb{C}\cong \tilde{w}(\tilde{X}) \to w(X)$ is injective on open squares of side length $2\lambda \epsilon_0$ (in coordinates $x+i y$ on $\mathbb{C}$). Take $\varepsilon <\varepsilon_0$. Choose the open cover $\mathcal{A}_{\varepsilon}$ of $w(X)$ that consists of projections to $w(X)$ of all open squares in $\mathbb{C}$ 
of side length $\epsilon$ with sides parallel to the axes.
The minimal cardinality $\mathcal{N}(\mathcal{A}_{\varepsilon})$ of a subcover of $\mathcal{A}_{\varepsilon}$ does not exceed the minimal cardinality of a cover of the closure of a fundamental domain on $\tilde X$ by squares of side length $\varepsilon$. This number does not exceed  $c_+ \,\varepsilon^{-2}$ for a positive number $c_+$. On the other hand
$\mathcal{N}(\mathcal{A}_{\varepsilon})$ is bounded from below by the maximal number of disjoint squares of side length $\varepsilon$ that can be put in the interior of a fundamental domain on $\tilde X$. This number is bounded from below by $c_- \,\varepsilon^{-2}$ for a positive constant $c_-$.

The cover
$\varphi_0^{-1}(\mathcal{A}_{\varepsilon})$  consists of the projection to $w(X)$ of all rectangles in $\tilde{w}(\tilde{X})$ with sides parallel to the axes of horizontal side length $\lambda^{-1} \epsilon$ and vertical side length $\lambda\epsilon$. If such a rectangle  in $\tilde{w}(\tilde{X})$ intersects an $\varepsilon$-square there, then by the choice of $\varepsilon$ the projection is injective on the union of the square and the rectangle. Hence, the cover $\mathcal{A}_{\varepsilon} \vee \varphi_0^{-1}(\mathcal{A}_{\varepsilon})$ consists of the projection to $w(X)$ of all rectangles in $\tilde{w}(\tilde{X})$ with sides parallel to the axes of horizontal side length at most $\lambda^{-1} \epsilon$ and vertical side length at most $\epsilon$.

By induction $\mathcal{A}_{\varepsilon} \vee \varphi_0^{-1}(\mathcal{A}_{\varepsilon}) \vee \ldots \vee \varphi_0^{-n}(\mathcal{A}_{\varepsilon})$ consists of projections to $w(X)$ of all rectangles in $\tilde{w}(\tilde X)$ of horizontal side length at most $\lambda^{-n} \epsilon$ and vertical side length at most $\epsilon$.

Indeed, if this is proved for a natural number $n$ it is obtained for $n+1$ as follows.  The cover $\varphi_0^{-1}(\mathcal{A}_{\varepsilon}) \vee \ldots \vee \varphi_0^{-n-1}(\mathcal{A}_{\varepsilon})$ is the image under $(\varphi_0)^{-1}$ of the cover $\mathcal{A}_{\varepsilon} \vee \ldots \vee \varphi_0^{-n}(\mathcal{A}_{\varepsilon})$. Hence, it consists of rectangles of horizontal side length at most $\lambda^{-n-1} \varepsilon$ and vertical side length at most $\lambda \varepsilon$. The cover $\mathcal{A}_{\varepsilon}\vee \varphi_0^{-1}(\mathcal{A}_{\varepsilon}) \vee \ldots \vee \varphi_0^{-n-1}(\mathcal{A}_{\varepsilon})$ consists of intersections of such rectangles with $\varepsilon$-squares.

For each $n$ a lift of each set of  a minimal subcover of  $\mathcal{A}_{\varepsilon}$ can be covered by $ \lambda^n+1$ rectangles of horizontal side length $\lambda^{-n} \epsilon$ and vertical side length at most $\epsilon$. On the other hand at least $ \lambda^n$ rectangles of horizontal side length at most $\lambda^{-n} \epsilon$ and vertical side length at most $ \epsilon$ are needed to cover a lift a single $\varepsilon$-square of  a minimal subcover of  $\mathcal{A}_{\varepsilon}$.

We obtain
\begin{equation}\label{eq3.133} \lambda^n \leq \mathcal{N}(\mathcal{A}_{\varepsilon} \vee \varphi_0^{-1}(\mathcal{A}_{\varepsilon}) \vee \ldots \vee \varphi_0^{-n}(\mathcal{A}_{\varepsilon})) \leq c_+ \epsilon^{-2} (\lambda^n+1) \,.
\end{equation}
Hence,
$$
h(\varphi_0,\mathcal{A}_{\varepsilon}) = \limsup_{n \to \infty} \frac{1}{n}
\log \mathcal{N}(\mathcal{A}_{\varepsilon} \vee \varphi_0^{-1}(\mathcal{A}_{\varepsilon}) \vee \ldots \vee \varphi_0^{-n}(\mathcal{A}_{\varepsilon}))= \log \lambda\,.
$$
For any decreasing sequence of small numbers $\varepsilon_n$ the sequence of coverings $\mathcal{A}_{\varepsilon_n}$ is refining. Indeed, it is clear that
$\mathcal{A}_{\varepsilon_1}\prec \mathcal{A}_{\varepsilon_2}$ for $\varepsilon_1< \varepsilon_2$. Any open cover $\mathcal{B}$ has a finite subcover. For small enough $\varepsilon>0$ each set in  $\mathcal{A}_{\varepsilon}$ is contained in a set of the finite subcover, hence, $\mathcal{A}_{\varepsilon}\prec \mathcal{B}$.

Hence,
$$
h(\varphi_0)= \log \lambda\,.
$$
The proposition is proved. \hfill $\Box$

\bigskip

We will now prove Theorem \ref{thm3.2} in the general case. In the rest of the section $X$ will be a closed connected Riemann surface
of genus $g$ with a set $E_m$ of $m\geq 0$ distinguished points, $3g-3+m>
0$. Let $\varphi_0$ be an absolutely extremal self-homeomorphism of $X$
with set of distinguished points $E_m$. Suppose $\varphi_0$ is not conformal and has quasiconformal distortion $K(\varphi_0)=\lambda^2>1$.
Then there exists a quadratic differential $\phi$ on $X$ with the following properties. $\varphi_0$ maps the set of critical points of $\phi$ to itself, it maps horizontal leaves of $\phi$ to horizontal leaves and vertical leaves to vertical leaves.

Let $\varepsilon>0$ be a small number. We prepare the definition of an open cover $\mathcal{A}_{\varepsilon}$ of $X$. We will call an open $\phi$-rectangle
with both side lengths equal to $\varepsilon$ an open $\varepsilon$-square.  \index{$\varepsilon$-square}
An open $\phi$-rectangle with length of the horizontal sides equal to $\varepsilon_1$ and length of the vertical sides equal to $\varepsilon_2$ will be called an open $(\varepsilon_1, \varepsilon_2)$-rectangle.

An $\varepsilon$-star at a singular point of $\phi$ is defined as follows. \index{$\varepsilon$-star}
Consider a simply connected neighbourhood $U$ of a singular point of $\phi$ and
distinguished coordinates $z$ on it in which $\phi$ has the form $\phi(z)= (\frac{a+2}{2})^2 z^{a} (dz)^2$ where $a $ is the order of the singular point. Since there are only finitely many singular points, we may assume that in these coordinates $U$ has the form $\{|z|< c\} $ for a constant $c$ that does not depend on the singular point. The horizontal bisectrices  $\{r \exp(\frac{2\pi j i}{a+2}),\, 0<r<c\} $, 
$j=0,\ldots,a+1, $
are segments of horizontal trajectories, the vertical bisectrices $\{r \exp(\frac{2\pi (j+\frac{1}{2}) i}{a+2}),\, 0<r<c\} $, 
$j=0,\ldots,a+1, $ are segments of vertical trajectories. \index{bisectrix ! horizontal} \index{bisectrix ! vertical} \index{halfsector}
 Consider the ''half-sectors'' between a horizontal bisectrix and a nearest vertical bisectrix. There are $2(a+2)$ such ''half-sectors''. They have the form
$\mathfrak{s}_j= \{r \exp({ x i}): \, x \in  (\frac{2\pi j i}{2(a+2)}, \frac{2\pi (j+1) i}{2(a+2)}), \;0<r<c \} \,, j=0,\ldots,2(a+2)-1$. \index{$\mathfrak{s}_j$} \index{$\mathfrak{s}_j^h$}
Denote by $\mathfrak{s}_j^h\stackrel{def}={\rm Int}(\overline{\mathfrak{s}_{2j}\cup \mathfrak{s}_{2j+1}})$ the sectors between two horizontal bisectrices.

Each half-sector $\mathfrak{s}_j$ is taken by
a branch $g_j$ of the mapping $z \to z^{\frac{a+2}{2}}$ conformally onto a quarter-disc with center $0$ in one of the quarters $\{\pm\mbox{Im}z >0, \pm\mbox{Re}z>0\}$ of the complex plane. For $\sqrt{2}\varepsilon<c^{\frac{a+2}{2}}$ we intersect the quarter disc with an open square in the complex plane with center $0$ and side length $2 \varepsilon$. We obtain an open square $Q_j(\varepsilon)$ of side length $\varepsilon$. The union of the preimages $g_j^{-1}(Q_j(\varepsilon)), \, j=1,\ldots, 2(a+2)-1, $ with the $2(a+2)$ open segments of $\phi$-length $\varepsilon$ emerging from the singular point and contained in a bisectrix, is a punctured neighbourhood of the singular point. Its union with the singular point is the required  $\varepsilon$-star at the singular point (see Figure \ref{figen.1}). \index{$g_j$}

The $\varepsilon$-star is bounded by $a+2$ horizontal segments of $\phi$-length $2 \varepsilon$ 
and  $a+2$ vertical segments of $\phi$-length $2 \varepsilon$
The distance in the $\phi$-metric from the center of the $\varepsilon$-star to the horizontal and vertical sides in the boundary of the star is equal to $\varepsilon$.
Hence, the $\phi$-distance from the center of an $\varepsilon$-star to its boundary equals $\varepsilon$. The $\varepsilon$-star is contained in a $\sqrt{2}\varepsilon$-neighbourhood (in the $\phi$-metric) of its center.

The definition of a star at a distinguished point implies immediately that for any positive $\varepsilon< \frac{1}{\sqrt{2}}c^{\frac{a+2}{2}}$ any $\varepsilon$-star punctured at its singular point of order $a$ can be covered by $4(a+2)$ $\varepsilon$-squares contained in the  $\varepsilon$-star. Indeed, the part of the star contained in the half-sectors $\mathfrak{s}_j$ is covered by $2(a+2)$ $\varepsilon$-squares, and each bisectrix is covered by an open $\varepsilon$-square. The latter follows from the fact that e.g. each horizontal bisectrix is contained in the sector between two neighbouring vertical bisectrices, and a branch of the mapping $z \to z^{\frac{a+2}{2}}$ takes this sector conformally onto a half-disc with center $0$ in the complex plane.
\begin{figure}[h]
\begin{center}
\includegraphics[width=100mm]{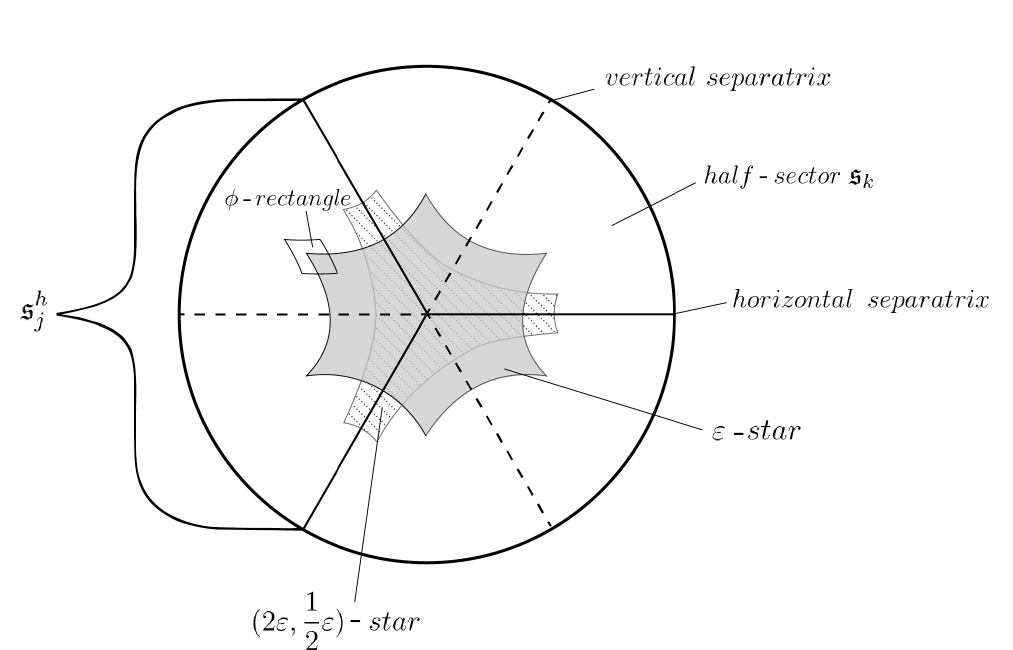}
\end{center}
\caption{An $\varepsilon$-star and a $(2\varepsilon, \frac{\varepsilon}{2})$-star at a $\phi$-singular point}
\label{figen.1}
\end{figure}
An $(\varepsilon_1, \varepsilon_2)$-star at a singular point is defined similarly. Instead of intersecting
quarter-discs with center $0$ in the complex plane
with open squares (in Euclidean coordinates) of center $0$ and side length $2\varepsilon$ we intersect the quarter-discs with open rectangles  with center $0$, horizontal side length $2\varepsilon_1$ and vertical side length $2\varepsilon_2$, and take the preimage under the chosen branch of the mapping $z \to z^{\frac{a+2}{2}}$.
The obtained intersections are $\phi$-rectangles of horizontal side length $\varepsilon_1$ and vertical side length $\varepsilon_2$. (See Figure \ref{figen.1}).

\begin{lemm}\label{lem3.109a}
Let $\varepsilon_1$ and $\varepsilon_2$ be positive numbers such that
$\sqrt{2} (\varepsilon_1 +\varepsilon_2) < c^{\frac{a+2}{2}}$ for the number c introduced above. If a $\phi$-rectangle with both side length's not exceeding  $\varepsilon_2$ intersects an $\varepsilon_1$-star  $\mathfrak{S}_{\varepsilon_1}$,  
then
the $(\varepsilon_1+ \varepsilon_2)$-star $\mathfrak{S}_{\varepsilon_1+ \varepsilon_2}$ at the singular point of  $\mathfrak{S}_{\varepsilon_1}$  exists, and
the rectangle is contained in the intersection of $\mathfrak{S}_{\varepsilon_1+ \varepsilon_2}$ with
a sector between two consecutive horizontal bisectrices (or two consecutive vertical bisectrices).
\end{lemm}
\index{$\mathfrak{S}_{\varepsilon}$}
\noindent{\bf Proof.} Since the distance (in the $\phi$-metric) of each point of the $\varepsilon$-star from the center does not exceed $\sqrt{2}\varepsilon_1$ and the diameter of the rectangle does not exceed $ \sqrt{2} \varepsilon_2$, the rectangle is contained in the $\sqrt{2}(\varepsilon_1+\varepsilon_2)$-neighbourhood (in distinguished coordinates) of the distinguished point. 
Suppose the rectangle intersects a vertical bisectrix 
at a point $x$.
Let $\mathfrak{s}_j^h$ be the sector between two neighbouring horizontal bisectrices that contains the vertical bisectrix. The rectangle can be obtained as follows. Take the maximal open vertical segment $s$ 
through $x$ that is contained in it. It has $\phi$-length at most $\varepsilon_2$. For each point in $s$  we take the maximal open horizontal segment 
through this point that is contained in the rectangle. Each has $\phi$-length at most $\varepsilon_2$.  Notice that the open rectangle does not contain singular points. Take the union of all these segments.
The points on these segments are reachable starting from the singular point along a piece of a vertical trajectory of length at most $\varepsilon_1 + \varepsilon_2$ followed by a piece of a horizontal trajectory of length at most $\varepsilon_2$. Hence, the rectangle is contained in the
$(\varepsilon_1 +\varepsilon_2)$-star. If the rectangle intersects a horizontal bisectrix or does not intersect any bisectrix, the proof is the same.
\hfill $\Box$

\begin{lemm}\label{lem3.109}
There exists a positive number $\delta_0$ such that there exist disjoint  $\delta_0$-stars at the singular points.
Moreover, there exists a positive number $\varepsilon_0< \delta_0$ such that for each $\varepsilon < \varepsilon_0$ and
for each point $x \in X$ which is not in an $\varepsilon$-star at a singular point 
there exists an $\varepsilon$-square with center $x$.
\end{lemm}
Recall that a square in $X$ is required to be contained in the $\phi$-regular part of $X$.

\medskip
\noindent {\bf Proof.} It is clear, that for some $\delta_0>0$ there exist disjoint $\delta_0$-stars at the singular points. For each point in the complement $X_{\delta_0}$ of all $\frac{1}{2}\delta_0$-stars there is a an $\varepsilon(x)$-square with center at this point for some positive number $\varepsilon(x)$ depending on $x$. Since $X_{\delta_0}$ is compact it can be covered by a finite number of $\frac{1}{2}\varepsilon(x_j)$-squares, $j=1,\ldots,N$. Let $\varepsilon_0$ be the minimum of the numbers $\frac{1}{2}\varepsilon(x_j)$. Then for each $x \in X_{\delta_0}$ there is an $\varepsilon_0$-square with center at this point. \index{$g_j^h$}

Let $x$ be in a $\delta_0$-star at a singular point but not in an $\varepsilon$-star. Suppose the distance of the point $x$ to a vertical bisectrix is equal to its distance to the union of all bisectrices.
Let $ \mathfrak{s}_j^h$ be the sector of the $\delta_0$-star, that is bounded by two consecutive horizontal bisectrices and contains the latter vertical bisectrix.
Let $g_j^h$ be a branch of the mapping $z \to z^{\frac{a+2}{2}}$ that maps $\mathfrak{s}_j^h$ to the upper or lower half-plane. It defines flat coordinates 
on $ \mathfrak{s}_j^h$.
The point $g_j^h(x)$ is contained in the intersection of the upper or lower half-plane with
a square of side length $2 \delta_0$ and center $0$, but not in the square of side length $2\varepsilon$ and center $0$. Moreover, the distance of $g_j^h(x)$ to the imaginary axis does not exceed its distance to the real axis.
Hence, the intersection of the half-plane with the square of side length $2\delta_0$ and center $0$ contains a square of side length $\varepsilon$ with center $g_j^h(x)$.
The existence of an $\varepsilon$-square with center $x$ is proved.
\hfill $\Box$

\bigskip

We prove now the following proposition which is one of the estimates of Theorem \ref{thm3.2}.

\begin{prop}\label{prop3.110}
Let $X$ be a (closed connected) Riemann surface
of genus $g$ with a set $E_m$ of $m\geq 0$ distinguished points, $3g-3+m>
0$. Let $\varphi_0 \in {\rm Hom} (X ; \emptyset , E_m)$ be a
non-periodic absolutely extremal self-homeomorphism
of $X$ with set of distinguished
points $E_m$. Then
$$
h(\varphi_0) \leq  \frac{1}{2} \log K(\varphi_0)
$$
\end{prop}

\noindent{\bf Proof.} Let $\varepsilon_0$ be as in Lemma \ref{lem3.109}.
Take a positive number $\varepsilon< \min\{\frac{\varepsilon_0}{\lambda}, \frac{c^{\frac{a+2}{2}}}{4\sqrt{2}\lambda}\}$. Let $\mathcal{A}_{\varepsilon}$ be the set which consists of all $\varepsilon$-squares and all $\varepsilon$-stars at distinguished points.
Then $\mathcal{A}_{\varepsilon}$ is an open cover of $X$.

The cover $\varphi_0^{-1}(\mathcal{A}_{\varepsilon})$ consists of all $(\lambda^{-1} \varepsilon, \lambda \varepsilon)$-stars at distinguished points and all $(\lambda^{-1} \varepsilon, \lambda \varepsilon)$-rectangles. Indeed, the image under $\varphi_0^{-1}$ of an $\varepsilon$-star at a singular point of $\phi$ is a $(\lambda^{-1} \varepsilon, \lambda \varepsilon)$-star at some singular point of $\phi$. Vice versa, take a $(\lambda^{-1}\varepsilon, \lambda\varepsilon)$-star $\mathfrak{S}$. Its image $\varphi_0(\mathfrak{S})$
under $\varphi_0$ is an $\varepsilon$-star and $\mathfrak{S}$ is the image of the
$\varepsilon$-star $\varphi_0(\mathfrak{S})$ under $\varphi_0^{-1}$. The argument for rectangles is the same.

Each $(\lambda^{-1} \varepsilon, \varepsilon)$-rectangle is the intersection of an $\varepsilon$-square with a $(\lambda^{-1} \varepsilon, \lambda\varepsilon)$-rectangle. Indeed, if the midpoint of the $(\lambda^{-1} \varepsilon, \varepsilon)$-rectangle is outside the $\lambda\varepsilon$-stars, then by Lemma \ref{lem3.109} there exists a $\lambda\varepsilon$-square with center at this point. This square contains an $\varepsilon$-square and a $(\lambda^{-1} \varepsilon, \lambda\varepsilon)$-rectangle centered at this point. The intersection of the latter two sets is the required  $(\lambda^{-1} \varepsilon, \varepsilon)$-rectangle.

Suppose the center $x$ of a $(\lambda^{-1} \varepsilon, \varepsilon)$-rectangle $R$ is in a $\lambda\varepsilon$-star $\mathfrak{S}_{\lambda\varepsilon}$.
Lemma \ref{lem3.109a} with $\varepsilon_1=\lambda \varepsilon$, $\varepsilon_2=\varepsilon$,
and $\sqrt{2}(\varepsilon_1+\varepsilon_2)= \sqrt{2}(\lambda+1)\varepsilon<\frac{1}{2}c^{\frac{a+2}{2}}$ implies that the rectangle $R$ 
is contained in a sector of the $(\lambda+1)\varepsilon$-star $\mathfrak{S}_{(\lambda +1)\varepsilon}$ between consecutive horizontal bisectrices, or between consecutive vertical bisectrices.
We assume that $R$
is contained in a sector $\mathfrak{s}_j^h \cap \mathfrak{S}_{(\lambda +1)\varepsilon}$ of the $(\lambda+1)\varepsilon$-star between two consecutive horizontal bisectrices. The remaining case is treated similarly and is even slightly simpler.

By our choice of $\varepsilon$ the $2(\lambda+1)\varepsilon$-star $\mathfrak{S}_{2(\lambda +1)\varepsilon}$ at the given singular point exists. The image $g_j^h(\mathfrak{s}_j^h \cap \mathfrak{S}_{2(\lambda +1)\varepsilon})$ is the intersection of the upper or lower half-plane with
a Euclidean square of side length $4(\lambda+1)\varepsilon$. 
The image $g_j^h(R)$ is a Euclidean rectangle of horizontal side length $\lambda^{-1} \varepsilon$ and vertical side length $\varepsilon$ contained in
the Euclidean rectangle  $g_j^h(\mathfrak{s}_j^h \cap \mathfrak{S}_{2(\lambda +1)\varepsilon})$
of horizontal side length $4(\lambda+1) \varepsilon$ and vertical side length $2(\lambda+1) \varepsilon$ . Hence,
$g_j^h(R)$ can be written as intersection of
a Euclidean rectangle of side lengths  $\lambda^{-1} \varepsilon$ and $\lambda \varepsilon$, and a Euclidean $\varepsilon$-square, both contained in the Euclidean rectangle $g_j^h(\mathfrak{s}_j^h \cap \mathfrak{S}_{2(\lambda +1)\varepsilon})$.
Hence, $R$ can be written as the desired intersection.

We showed that each $(\lambda^{-1} \varepsilon, \varepsilon)$-rectangle
can be written as $A_0\cap \varphi_0^{-1}(A_1)\in\mathcal{A}_{\varepsilon} \vee \varphi_0^{-1}(\mathcal{A}_{\varepsilon})$ for $\varepsilon$-squares $A_0$ and $A_1$.

Any $(\lambda^{-n} \varepsilon, \varepsilon)$-rectangle can be written as
$A_0\cap \varphi_0^{-1}(A_1)\cap \ldots \cap \varphi_0^{-n}(A_n)\in
\mathcal{A}_{\varepsilon} \vee \varphi_0^{-1}(\mathcal{A}_{\varepsilon}) \vee \ldots \vee \varphi_0^{-n}(\mathcal{A}_{\varepsilon})$ for $\varepsilon$-squares $A_j$, $j=0,1,\ldots,n$. The proof goes by induction. Assume the claim is true for a natural number $n$. We prove that it is true for $n+1$. Suppose first the midpoint of the $(\lambda^{-n-1} \varepsilon, \varepsilon)$-rectangle $A$ is outside the $\lambda\varepsilon$-stars. Then by Lemma \ref{lem3.109} the $\lambda \varepsilon$-square around this point exists.
Hence, there is an $\varepsilon$-square $A'$ and a $(\lambda^{-n-1} \varepsilon, \lambda\varepsilon)$-rectangle $A''$ with the same center so that
the intersection $A'\cap A''$ is equal to $A$. 
The image $\varphi_0(A'')$ is a $(\lambda^{-n} \varepsilon, \varepsilon)$-rectangle which is an element of
$\mathcal{A}_{\varepsilon} \vee \varphi_0^{-1}(\mathcal{A}_{\varepsilon}) \vee \ldots \vee \varphi_0^{-n}(\mathcal{A}_{\varepsilon})$
by the claim for the number $n$.
Hence, its preimage $A''$ under $\varphi_0$ is in $\varphi_0^{-1}(\mathcal{A}_{\varepsilon}) \vee \varphi_0^{-2}(\mathcal{A}_{\varepsilon}) \vee \ldots \vee \varphi_0^{-1-n}(\mathcal{A}_{\varepsilon})$. Since $A=A' \cap A''$ the claim is proved in this case.

In the remaining case the midpoint of the $(\lambda^{-n-1} \varepsilon, \varepsilon)$-rectangle $A$ is inside a $\lambda\varepsilon$-star. Using the proof for $n=1$ and the induction argument above
we obtain the claim for $n+1$ also in this case.

Each $(\lambda^{-n} \varepsilon, \varepsilon)$-star can be written as $A_0\cap \varphi_0^{-1}(A_1)\cap\ldots \cap \varphi_0^{-n}(A_n)\in  \mathcal{A}_{\varepsilon} \vee \varphi_0^{-1}(\mathcal{A}_{\varepsilon}) \vee \ldots \vee \varphi_0^{-n}(\mathcal{A}_{\varepsilon})$ for $\varepsilon$-stars $A_j$, $j=0,\ldots ,n$. This is true for $n=1$
since each $(\lambda^{-1} \varepsilon, \varepsilon)$-star is the intersection of a $(\lambda^{-1} \varepsilon, \lambda \varepsilon)$-star with an $\varepsilon$-star.
We suppose that the claim is true for $n$, and prove it for $n+1$.
Each  $(\lambda^{-n-1} \varepsilon, \varepsilon)$-star is the intersection of
an $\varepsilon$-star with a
$(\lambda^{-n-1} \varepsilon, \lambda\varepsilon)$-star. The latter is a preimage under $\varphi_0$ of a $(\lambda^{-n} \varepsilon, \varepsilon)$-star, hence, by the induction hypothesis it can be written as  $ \varphi_0^{-1}(A_1)\cap \ldots \cap\varphi_0^{-n-1}(A_{n+1}) \in \varphi_0^{-1}(\mathcal{A}_{\varepsilon}) \vee \ldots \vee \varphi_0^{-n-1}(\mathcal{A}_{\varepsilon})$. This proves the claim.

Cover $X$ by a finite subcover of $\mathcal{A}_{\varepsilon}$. We may assume that this subcover consists of all $\varepsilon$-stars and a finite number of $\varepsilon$-squares that cover the complement of the singular points in $X$. Let $C_1$ be the number of squares in the subcover, and $C_2$ the number of stars . Each square can be covered by no more than $\lambda ^n+1$ sets that are
$(\lambda^{-n} \varepsilon, \varepsilon)$-rectangles. These rectangles are in $\mathcal{A}_{\varepsilon} \vee \varphi_0^{-1}(\mathcal{A}_{\varepsilon}) \vee \ldots \vee \varphi_0^{-n}(\mathcal{A}_{\varepsilon})$.
Each singular point can be covered by a $(\lambda^{-n} \varepsilon, \varepsilon)$-star which is also contained in $\mathcal{A}_{\varepsilon} \vee \varphi_0^{-1}(\mathcal{A}_{\varepsilon}) \vee \ldots \vee \varphi_0^{-n}(\mathcal{A}_{\varepsilon})$. We obtain a subcover of cardinality at most $C_1( \lambda ^n+1) +C_2  \leq C ( \lambda ^n+1)$.
Hence, for any $\varepsilon < \varepsilon_0$
\begin{equation}\label{eq3.134}
h(\varphi_0,\mathcal{A}_{\varepsilon}) \leq \limsup \frac{1}{n} \log(C(\lambda^n+1))=\log \lambda\,.
\end{equation}
Take a sequence of $\varepsilon$'s decreasing to $0$. We get \eqref{eq3.134} for a refining sequence of coverings. Hence,
\begin{equation}
h(\varphi_0)\leq \log \lambda= \frac{1}{2} \log K(\varphi_0)\,. \nonumber
\end{equation}
The proposition is proved.  \hfill $\Box$

\bigskip
The following proposition states the opposite inequality.

\bigskip
\begin{prop}\label{prop3.111}
Let $X$ be a (closed connected) Riemann surface
of genus $g$ with a set $E_m$ of $m\geq 0$ distinguished points, $3g-3+m>
0$. Let $\varphi_0 \in {\rm Hom} (X ; \emptyset , E_m)$ be a
non-periodic absolutely extremal self-homeomorphism of $X$ with set of distinguished
points $E_m$. Then
$$
h(\varphi_0) \geq  \frac{1}{2} \log K(\varphi_0)
$$
\end{prop}

\bigskip

\noindent{\bf Proof.} Let
$\varepsilon< \min\{\frac{\varepsilon_0}{3\lambda+1}, \frac{c^{\frac{a+2}{2}}}{4\sqrt{2}\lambda}\}$.
We consider the cover $ \mathcal{B}_{\varepsilon}$ which consist of
all $\varepsilon$-squares and all $\frac{\varepsilon}{2}$-stars at singular points.
We prove first that each non-empty set of the cover  $\mathcal{B}_{\varepsilon} \vee\varphi_0^{-1}(\mathcal{B}_{\varepsilon})$ is either a $(\lambda^{-1}\frac{\varepsilon}{2}, \frac{\varepsilon}{2})$-star at a singular point of $\phi$ 
or an $(\varepsilon_1,\varepsilon_2)$-rectangle with $\varepsilon_1\leq \lambda^{-1}\varepsilon$ and
$\varepsilon_2\leq \varepsilon$.

We have to consider non-empty intersections
$A_1 \cap \varphi_0^{-1}(A_2)$ for two sets in  $\mathcal{B}_{\varepsilon}$.
Suppose both, $A_1$ and $A_2$ are $\varepsilon$-squares, hence, $\varphi_0^{-1}(A_2)$ is an $(\lambda^{-1}\varepsilon, \lambda\varepsilon)$-rectangle.
If the center $x_0$ of $A_1$ is not in a $2\lambda\varepsilon$-star, then by Lemma \ref{lem3.109} there exists a $2\lambda\varepsilon$-square with center at the same point $x_0$. $\varphi_0^{-1}(A_2)$ is a $(\lambda^{-1}\varepsilon, \lambda \varepsilon)$-rectangle and intersects $A_1$, hence $A_1\cap \varphi_0^{-1}(A_2)$ is contained in the $2\lambda \varepsilon$-square centered $x_0$. It is now clear that $A_1 \cap \varphi_0^{-1}(A_2)$ is an $(\varepsilon_1,\varepsilon_2)$-rectangle with $\varepsilon_1\leq \lambda^{-1}\varepsilon$ and
$\varepsilon_2\leq \varepsilon$.

Suppose $A_1$ and $A_2$ are $\varepsilon$-rectangles and the centre of $A_1$ is in a $2\lambda \varepsilon$-star. Then by Lemma \ref{lem3.109a} $A_1$ is contained in the $(2\lambda +1) \varepsilon$-star and $\varphi_0^{-1}(A_2)$ intersects the $(2\lambda +1) \varepsilon$-star. Hence, by the same Lemma \ref{lem3.109a} $\varphi_0^{-1}(A_2)$ is contained in the $(3\lambda +1) \varepsilon$-star. Each of the $A_1$ and $\varphi_0^{-1}(A_2)$ is contained in a sector between two consecutive horizontal or between two consecutive vertical bisectrices. Moreover, the two sectors intersect at least along a half-sector. The union of the two sectors, equipped with flat coordinates, is an open subset of
a half-plane or a three-quarter plane with Euclidean coordinates. It is now clear that the intersection of $A_1$ and $\varphi_0^{-1}(A_2)$ is an $(\varepsilon_1,\varepsilon_2)$-rectangle with $\varepsilon_1\leq \lambda^{-1}\varepsilon$ and
$\varepsilon_2\leq \varepsilon$.

If $A_1$ is an $\frac{\varepsilon}{2}$-star, further $A_2$ is an $\varepsilon$-square, and $A_1\cap \varphi_0^{-1}(A_2)$ is non empty, then by Lemma \ref{lem3.109a} $A_1\cup \varphi_0^{-1}(A_2)$ is contained in the $(\frac{\varepsilon}{2} +\lambda \varepsilon)$-star. Moreover, $\varphi_0^{-1}(A_2)$ is contained in a sector of this star between two consecutive horizontal or two consecutive vertical bisectrices. It is now clear that the intersection of $A_1$ and $\varphi_0^{-1}(A_2)$ is an $(\varepsilon_1,\varepsilon_2)$-rectangle with $\varepsilon_1\leq \lambda^{-1}\varepsilon$ and
$\varepsilon_2\leq \varepsilon$.

Let $A_1$ be an $\varepsilon$-square and let $A_2$ be an  $\frac{\varepsilon}{2}$-star. Then $\varphi_0^{-1}(A_2)$ is a $(\lambda^{-1}\frac{\varepsilon}{2}, \lambda\frac{\varepsilon}{2})$-star.
If the intersection $A_1\cap \varphi_0^{-1}(A_2)$
is non-empty, then by Lemma \ref{lem3.109a}
$A_1$ is contained in a sector of a $(\frac{\lambda}{2}+1) \varepsilon$-star between two consecutive separatrices of the same kind (both horizontal or both vertical) with the same center. In flat coordinates on this set we obtain the intersection of an $\varepsilon$-square
with an $(\varepsilon_1,\varepsilon_2)$-rectangle with $\varepsilon_1\leq \lambda^{-1}\frac {\varepsilon}{2}$ and $\varepsilon_2\leq \lambda \frac{\varepsilon}{2}$. The claim is obtained also in this case.

If both, $A_1$ and $A_2$ are stars they have the same center if the intersection is not empty, and the claim is clear.

We prove now by induction that each  non-empty set of the form
\begin{equation}\label{eq3.135}
A_0 \cap \varphi_0^{-1}(A_1)\cap \ldots \cap \varphi_0^{-n}(A_n)
\end{equation}
(with $A_j, j=0,\ldots n,$ being sets in  $\mathcal{B}_{\varepsilon}$), is either an $(\varepsilon_1,\varepsilon_2)$-rectangle with $\varepsilon_1\leq \lambda^{-n}\varepsilon$ and $\varepsilon_2\leq \varepsilon$ or a
$(\lambda^{-n}\frac{\varepsilon}{2}, \frac{\varepsilon}{2})$-star at a singular point of $\phi$. Suppose this is true for a natural number $n$. Prove it for $n+1$.
We have to consider intersections $A_1 \cap \varphi_0^{-1}(B)$ with
$A_1 \in \mathcal{B}_{\varepsilon} $  and $B \in \mathcal{B}_{\varepsilon} \vee \varphi_0^{-1}(\mathcal{B}_{\varepsilon}) \vee \ldots \vee \varphi_0^{-n}(\mathcal{B}_{\varepsilon})$. Let $A_1$ be an $\varepsilon$-square. We have to intersect it either with an  $(\varepsilon_1,\varepsilon_2)$-rectangle with $\varepsilon_1\leq \lambda^{-n-1}\varepsilon$ and $\varepsilon_2\leq \lambda\varepsilon$, or with
a $(\lambda^{-n-1}\frac{\varepsilon}{2}, \frac{\varepsilon}{2})$-star at singular a point of $\phi$. In the same way as before we obtain that the intersection is a $(\varepsilon_1,\varepsilon_2)$-rectangle with $\varepsilon_1\leq \lambda^{-n-1}\varepsilon$ and $\varepsilon_2\leq \varepsilon$.

Let $A_1$ be an $\frac{\varepsilon}{2}$-star at a singular point of $\phi$. We have to intersect it either with a $(\lambda^{-n-1}\frac{\varepsilon}{2},\lambda \frac{\varepsilon}{2})$-star at the same singular point, or with an $(\varepsilon_1,\varepsilon_2)$-rectangle with $\varepsilon_1\leq \lambda^{-n-1}\varepsilon$ and $\varepsilon_2\leq \lambda\varepsilon$. In the first case we obtain a $(\lambda^{-n-1}\frac{\varepsilon}{2}, \frac{\varepsilon}{2})$-star. In the second case the rectangle is contained in a sector between two consecutive bisectrices of the same kind of a larger star and we obtain in flat coordinates the intersection of the $(\varepsilon_1,\varepsilon_2)$-rectangle with an
$(\varepsilon'_1,\varepsilon'_2)$-rectangle, where $\varepsilon'_1$ and $\varepsilon'_2$ do not exceed $\varepsilon$.
The smaller horizontal side length is $\lambda^{-n-1}\frac{\varepsilon}{2}$, the smaller vertical side length is $\varepsilon$. We proved the claim.

The proof of the proposition is completed as follows. Consider an $\varepsilon$-square $Q$ with $\phi$-distance bigger than $2 \lambda\varepsilon$ from any singular point. Then for any non-negative integer $n$ no star in the cover $\mathcal{B}_{\varepsilon} \vee \varphi_0^{-1}(\mathcal{B}_{\varepsilon}) \vee \ldots \vee \varphi_0^{-n}(\mathcal{B}_{\varepsilon})$ intersects $Q$.
Fix a natural number $n$. We want to estimate $\mathcal{N}(\mathcal{B}_{\varepsilon} \vee \varphi_0^{-1}(\mathcal{B}_{\varepsilon}) \vee \ldots \vee \varphi_0^{-n}(\mathcal{B}_{\varepsilon}))$ from below.
The sets in the cover of $X$ that intersect $Q$ are $(\varepsilon_1,\varepsilon_2)$-rectangles with $\varepsilon_1\leq \lambda^{-n}\varepsilon$ and $\varepsilon_2\leq \varepsilon$. To cover $Q$ at least $\lambda^n$ sets are needed. Hence,
\begin{align}\label{eq3.136}
h(\varphi_0,\mathcal{B}_{ \epsilon})\geq & \limsup_{n \to \infty} \frac{1}{n} \log \mathcal{N}(\mathcal{B}_{\varepsilon} \vee \varphi_0^{-1}(\mathcal{B}_{\varepsilon}) \vee \ldots \vee \varphi_0^{-n}(\mathcal{B}_{\varepsilon}))\nonumber\\
\geq & \limsup_{n \to \infty} \frac{1}{n} \log \lambda^n =\lambda\,.
\end{align}
Take a sequence of $\varepsilon$'s decreasing to $0$ we obtain a refining sequence of coverings. We proved that $h(\varphi_0)\geq \lambda$. \hfill $\Box$

The first equality in Theorem \ref{thm3.2} is proved. The second one is Corollary \ref{corr2.1}. Theorem \ref{thm3.2} is proved. \hfill $\Box$

\medskip

\section[Pseudo-Anosov self-homeomorphisms  are entropy minimizing]    {Pseudo-Anosov self-homeomorphisms of closed surfaces are entropy minimizing}

\label{sec:entropy.2}
In this section we will prove Theorem \ref{thm3.1} along the lines of \cite{FLP}.
Let $X$ be a closed surface of genus at least two and let $x_0$ be a chosen base point of $X$. 
Let $f$ be an arbitrary self-homeomorphism of $X$.
If $f(x_0)=x_0$ the mapping $f$ induces a homomorphism $f_*:\pi_1(X,x_0)\toitself$. Indeed, for each loop $\gamma$ with base point $x_0$ the image
$f\circ \gamma$ is a loop with base point $x_0$, whose class in $\pi_1(X,x_0)$ depends only on the class of $\gamma$ in $\pi_1(X,x_0)$.
\index{$f_*$}

In the general situation the plan is the following.
We will show that 
the mapping $f$ induces a conjugacy class of homomorphisms $\widehat{f_{\#}}$
from the group ${\rm Deck}(\tilde{X},X)$ of covering transformations (identified with the fundamental group of $X$) into itself.
\index{$\widehat{f_{\#}}$}
We will define a quantity $\Gamma_{\widehat{f_{\#}}}$ that depends only on the conjugacy class $\widehat{f_{\#}}$ and provides a lower bound
for the entropy of $f$. Moreover, the quantity is more geometric than the entropy. In the case when $f$ is a non-periodic absolutely extremal homeomorphism
the quantity can be estimated from below by $\frac{1}{2}\log K(f)$.

Before coming to the fundamental group we consider any group $G$. Let $\mathcal{G}=\{g_1,\ldots,g_r\}$ be any subset of $G$ that generates $G$. For an element $g \in G$ we denote by $\mathcal{L}_{\mathcal{G}}(g)$ the minimal length
of a word
in the $g_j$'s and $g_j^{-1}$'s that represents $g$. (The length of a word in the $g_j$ is the sum of the absolute values of the powers of exponents of the $g_j$ that appear in the word.)
For any other subset $\mathcal{G}'= \{g_1', \ldots, g_{r'}'\}$ of $G$ that generates $G$ the inequality
$\mathcal{L}_{\mathcal{G}'}(g) \leq (\max \mathcal{L}_{\mathcal{G}'}(g_j))  \mathcal{L}_{\mathcal{G}}(g)$ holds.
\index{$\mathcal{G}$} \index{$\mathcal{L}_{\mathcal{G}}(g)$}

For a group homomorphism $A:G \to G$ we put
\begin{equation}\label{eq3.101}
\Gamma_A=\underset{g \in G} {\sup} \, \limsup_{n\to \infty} \frac{1}{n} \log \mathcal{L}_{\mathcal{G}}(A^n g)
\end{equation}
\index{$\Gamma_A$}
The following lemma is a straightforward consequence of the definition.
\begin{lemm}\label{lem3.101}
$$\Gamma_A=\underset{g_j \in \mathcal{G}} {\max} \, \limsup_{n\to \infty} \frac{1}{n} \log \mathcal{L}_{\mathcal{G}}(A^n g_j)\,.$$
Hence, $\Gamma_A$ is finite.
\end{lemm}

For an element $g \in G$ we consider the group homomorphism $gAg^{-1}:G \to G$ defined by $gAg^{-1}(x)=gA(x) g^{-1}$.

For convenience of the reader we formulate the following technical lemma from \cite{FLP} (see  Lemma 10.6 in \cite{FLP})). The proof is a straightforward calculation and is left to the reader.

\begin{lemm}\label{lem3.100}
Let $(a_{n})_{n=1}^{\infty}$ and $(b_{n})_{n=1}^{\infty}$ be two sequences of positive numbers. Then\\
$(1)$  $\limsup \frac{1}{n} \log(a_n+b_n) = \max(\limsup \frac{1}{n} \log(a_n),
\limsup \frac{1}{n} \log(b_n);$\\
$(2)$  $ \limsup \frac{1}{n} \log(a_n) \leq \limsup \frac{1}{n} \log(a_1+ \ldots + a_n) \leq \max(0, \limsup \frac{1}{n} \log(a_n)).$
\end{lemm}

The following proposition is proved in \cite{FLP} (see Proposition 10.5 of \cite{FLP} and its proof).

\begin{prop}\label{prop3.101'}
The equality
\begin{equation}
\Gamma_A= \Gamma_{gAg^{-1}}\nonumber
\end{equation}
holds.
\end{prop}

\noindent{\bf Proof.} By the definition we have $(gAg^{-1})(x)=gA(x)g^{-1},\,g\in G$. Hence,
$$(gAg^{-1})^2(x)=gA(gA(x)g^{-1}) g^{-1}=gA(g)A^2(x)A(g^{-1})g^{-1}.$$
By induction
$$(gAg^{-1})^n(x)=gA(g)\ldots A^{n-1}(g)A^n(x)A^{-n+1}(g)\ldots A(g^{-1})g^{-1}\,.$$
Hence, since $ \mathcal{L}_{\mathcal{G}}(y)= \mathcal{L}_{\mathcal{G}}(y^{-1})$ for each $y\in G$ we obtain
\begin{align}\label{eq3.101'''}
\Gamma_{gAg^{-1}}
= &\underset {x \in G}\sup \limsup \frac{1}{n} \log(\mathcal{L}_{\mathcal{G}}(g 
\ldots A^{n-1}(g)A^n(x)A(g^{-n+1})\ldots 
g^{-1})) \nonumber\\
\leq &\underset {x \in G}\sup \limsup  \frac{1}{n} \log\Big(2 \big ( \mathcal{L}_{\mathcal{G}}(g) +\ldots +  \mathcal{L}_{\mathcal{G}}(A^{n-1}(g))\big) 
+  \mathcal{L}_{\mathcal{G}}(A^n(x))\Big)\,.
\end{align}
If $A^{n_0}(g)$ is the identity for some $n_0$, then it follows from Statement (1) of Lemma \ref{lem3.100}, that
the right hand side of \eqref{eq3.101'''} does not exceed
$$\limsup \frac{1}{n}\log \mathcal{L}_{\mathcal{G}}(A^n(x)).$$
In the remaining case
$\mathcal{L}_{\mathcal{G}}(A^n(g))\geq1$ for each $n$, and by Statements (1) and (2) of Lemma \ref{lem3.100}
\begin{align*}
\limsup &\frac{1}{n}\log\Big(\mathcal{L}_{\mathcal{G}}(gAg^{-1})(x)\Big)\\
\leq \max\Big(\limsup &\frac{1}{n}\log \mathcal{L}_{\mathcal{G}}(A^n(g)),\limsup \frac{1}{n}\log \mathcal{L}_{\mathcal{G}}(A^n(x))\Big)\,.
\end{align*}
We obtain $\Gamma_{gAg^{-1}}\leq \Gamma_A$. The opposite inequality follows by symmetry. \hfill $\Box$

\medskip

Let $X$ be  a closed surface and $x_0 \in X$ a base point. Denote by ${\sf P}:\tilde X\to X$ the universal covering of $X$. Recall that each point of $\tilde X$ that projects to a point $x \in X$ can be identified with an element $e_{x_0,x} \in \pi_1(X;x_0,x)$ (see e.g. \cite{Fo} and Section \ref{sec:2.0}).
Here $\pi_1(X;x_0,x)$ denotes the set of homotopy classes of arcs in $X$ with initial point $x_0$ and terminal point $x$.
We write $\{e_{x_0,x}\}$ if we mean the point of $\tilde X$ associated to the element $e_{x_0,x} \in \pi_1(X;x_0,x)$. The projection assigns to each point $\{e_{x_0,x}\}$ the point $x$. \index{$\pi_1(X;f(x_0),f(x_0))$} \index{$\pi_1(X;x_0,x)$}

Let $f:X\to X$ be a continuous mapping. For each pair of points $x_1$, $x_2$ in $X$ the mapping $f$  induces a mapping $f_*:\pi_1(X;x_1,x)\to \pi_1(X;f(x_1),f(x))$.
Namely, $f_*( e_{x_1,x})$ is 
the class represented by $f\circ \gamma_{x_1,x}$ for a  curve $\gamma_{x_1,x}$ that
represents $\{e_{x_1,x}\}$.

Let $x_0\in X$ be the base point of $X$. A lift of $f$ to $\tilde X$ is a continuous mapping $\tilde f$ such that $f\circ {\sf P}=
{\sf P}\circ \tilde{f}$. A lift of $f$ to $\tilde X$ can be described as follows (see also \cite{Fo}). Choose and fix a homotopy class of curves $e_{x_0,f(x_0)}^0\in \pi_1(X;x_0,f(x_0))$ with fixed endpoints $x_0$ and $f(x_0)$. Put
\begin{equation}\label{eq3.07b}
\tilde f (\{e_{x_0,x}\}) =\{e_{x_0,f(x_0)}^0 \, f_*(e_{x_0,x})\} , \; \{e_{x_0,x}\} \in \tilde X\,.
\end{equation}
Here $e_{x_0,f(x_0)}^0 \, f(e_{x_0,x})$ is the homotopy class of arcs represented by first traveling along an arc representing $ e_{x_0,f(x_0)}^0$ and then along an arc representing the image
$f_*(e_{x_0,x})$ of $e_{x_0,x}$ under the induced mapping $f_*:\pi_1(X;x_0,x)\to \pi_1(X;f(x_0 ),f(x))$. All other lifts are obtained by replacing $ e_{x_0,f(x_0)}^0$ by another element of $\pi_1(X; x_0, f(x_0))$.
Equivalently, for two lifts $\tilde{f_1}$ and $\tilde{f_2}$ of $f$ there exists a covering transformation $\theta \in \pi_1(X,x_0)$ such that $\tilde{f_1}= \theta \circ\tilde{f_2}$.
Indeed, if ${\sf P}\circ \tilde{f}_1 ={\sf P}\circ \tilde{f}_2 =f \circ {\sf P}$, then for each $\tilde{x}\in \tilde{X}$ the points $\tilde{f}_1(\tilde{x})$ and   $\tilde{f}_2(\tilde{x})$ differ by a covering transformation. The covering transformations depend continuously on $\tilde{x}$, hence, there is a single covering transformation $\theta$ such that  $\tilde{f_1}= \theta \circ\tilde{f_2}$.

The following lemma gives a lower bound for the entropy of a mapping $f$ in terms of metric properties of the lift $\tilde f$ of $f$.

\begin{lemm}\label{lem3.106} Let $X$ be a closed surface with metric $d$, $\tilde X$ its universal covering with the lifted metric also denoted by $d$, $f:X \to X$ a self-homeomorphism of $X$ and $\tilde f$ a lift of $f$  to $\tilde X$.
Then for any points $\tilde{x},\tilde{y} \in \tilde X$
\begin{equation}\label{eq3.108}
\limsup_{n \to \infty}  \frac{1}{n} d(\tilde f^n(\tilde{x}), \tilde f^n(\tilde{y})) \leq h(f).
\end{equation}
\end{lemm}

\noindent {\bf Proof.} Take any curve $\tilde \gamma$ in $\tilde X$ joining $\tilde{x}$ and $\tilde{y}$.  Let $\gamma$ be the projection of $\tilde \gamma$ to $X$. Since $\tilde \gamma$ is a compact subset of  $\tilde X$, there is a finite constant $l$ such that each point of $\gamma$ is covered at most $l$ times by points in $\tilde \gamma$. For small $\varepsilon >0$ we denote by $\mathcal{A}_{\varepsilon}$ the cover of $X$ by all discs of radius $\varepsilon$ in the metric $d$. For any decreasing sequence of $\varepsilon$'s we obtain a refining sequence of covers. Fix any small $\varepsilon>0$ and any positive number $\delta$. If $n$ is large there is a cover $\mathcal{A}_{\varepsilon,n}$ of $X$ by at most $\exp(n (h(f,\mathcal{A}_{\varepsilon} ) + \delta))$ elements of the cover $\mathcal{A}_{\varepsilon}
\vee f^{-1}(\mathcal{A}_{\varepsilon} ) \vee \ldots \vee f^{-n}(\mathcal{A}_{\varepsilon} )$. Note that $\mathcal{A}_{\varepsilon,n}$ is a subcover of $\mathcal{A}_{\varepsilon}
\vee f^{-1}(\mathcal{A}_{\varepsilon} ) \vee \ldots \vee f^{-n}(\mathcal{A}_{\varepsilon} )$.
For each element of the subcover $\mathcal{A}_{\varepsilon,n}$
there is an $\varepsilon$-disc $A_{\varepsilon,n}$ such that the element is contained in  $f^{-n}(A_{\varepsilon,n})$. Since $f^{-n}$ is a homeomorphism we see that $f^{n}(\gamma)$ can be covered by $\exp(n (h(f,\mathcal{A}_{\varepsilon}) + \delta))$ $\varepsilon$-discs.

The mapping $(\tilde f)^n$ is a lift of $f^n$ to $\tilde X$. Indeed,
$\tilde f$ lifts $f$, i.e. $p\circ \tilde{f}(\tilde{x})=f\circ p(\tilde{x})$ for $\tilde{ x}\in \tilde X$. By induction $p\circ   \tilde{f}^{n+1}(\tilde{x})=p\circ \tilde{f}^n(  \tilde{f}(\tilde{x}))=
f^n\circ p( \tilde{f}(\tilde{x}))=f^{n+1}\circ p(\tilde{x})$ for $\tilde{x}\in\tilde X$.
The curve $\tilde f ^n (\tilde\gamma)$ is a lift of $f^n(\gamma)$ to $X$. Hence,
each point of the curve $\tilde f ^n (\tilde \gamma)$ is covered by a lift of one of the $\varepsilon$-discs. Since each point of $f^n(\gamma)$ is the projection of at most $l$ points of $\tilde f ^n (\tilde \gamma)$ the latter curve can be covered by $l \exp(n (h(f,\mathcal{A}_{\varepsilon}) + \delta))$ $\varepsilon$-discs. This implies that 
the distance between its two endpoints $d(\tilde f ^n(x),\tilde f ^n(y))$, does not exceed $2 \varepsilon l  \exp(n (h(f,\mathcal{A}_{\varepsilon}) + \delta))$. Hence,
\begin{equation}\label{eq3.109}
\limsup_{n \to \infty} \frac{1}{n} \log d(\tilde f ^n(x),\tilde f ^n(y)) \leq h(f,\mathcal{A}_{\varepsilon}) +\delta.
\end{equation}
Since $\delta >0$ is arbitrary, equation \eqref{eq3.109} holds with $h(f,\mathcal{A}_{\varepsilon}) +\delta$ replaced by $h(f,\mathcal{A}_{\varepsilon})$. For a sequence of $\varepsilon$'s we obtain a refining sequence of coverings. The statement of
the lemma follows. \hfill $\Box$

\medskip
We will consider now a lift $\tilde{f}$ of $f$, and associate to it an action $\tilde{f}_{\#}$ on the fundamental group.
Recall that we associated to each element $\alpha$ of the 
fundamental group $\pi_1(X,x_0)$ a covering transformation as follows. We take a point $\tilde{ x}_0\in \tilde X$ that projects to $x_0$. Consider the lift $\tilde \alpha$ of a representative of $\alpha$ with initial point $\tilde{x}_0$. Associate to $\alpha$ the covering transformation $({\rm Is}^{\tilde{x}_0})^{-1}(\alpha)$ that maps $\tilde{x}_0$ to the terminal point of $\tilde\alpha$. The point $\tilde{x}_0$ will be fixed in this section. We denote the covering transformation $({\rm Is}^{\tilde{x}_0})^{-1}(\alpha)$ assocaited to $\alpha\in \pi_1(X,x_0)$ by $\alpha^{cov}$. Recall that $(\alpha_1\alpha_2)^{cov}= \alpha_2^{cov}\alpha_1^{cov}$. 
The covering transformation $\alpha^{cov}$, acts
as follows. Write the elements of $\tilde{X}$ as $\tilde{x}\cong \{e_{x_0,x}\}$. Then
\begin{equation}\label{eq3.07c}
\alpha^{cov}(\{e_{x_0,x}\}) = \{\alpha \, e_{x_0,x}\},\, \{ e_{x_0,x}\}\in \tilde{ X}\, .
\end{equation}
Fix a lift $\tilde{f}$ of $f$. Since  for each $\alpha \in \pi_1(X,x_0)$ the mapping $\tilde{f} \circ\alpha^{cov}$ is a lift of $f$, there is an element $(\tilde{f})_{\#}(\alpha) \in \pi_1(X,x_0)$ depending on $\alpha$ such that
\begin{equation}\label{eq3.07f}
\tilde f \circ\alpha^{cov} = \big((\tilde{f})_{\#}(\alpha)\big)^{cov} \circ\tilde f \,. 
\end{equation}
Since
\begin{align*}
\tilde f  \circ(\alpha_1\alpha_2)^{cov}(\tilde x)= & \tilde f  \circ (\alpha_2^{cov}\circ\alpha_1^{cov})(\tilde x)\\
= \big((\tilde{f})_{\#}(\alpha_2)\big)^{cov}\circ\tilde{f}(\alpha_1^{cov}(\tilde x))= & \Big(\big((\tilde{f})_{\#}(\alpha_2)\big)^{cov}\circ \big((\tilde{f})_{\#}(\alpha_1)\big)^{cov}\Big)\circ\tilde{f}(\tilde{x})
\end{align*}
the mapping $\alpha\to (\tilde{f})_{\#}(\alpha)$ is a group isomorphism of $\pi_1(X,x_0)$.

This isomorphism can be given explicitly as follows.
Let $\tilde f$ be given by \eqref{eq3.07b}.
For any element $\alpha \in \pi_1(X,x_0)$ we get the equality
\begin{align}\label{eq3.07''}
\tilde{f}\circ\alpha^{cov}(\tilde{x})=
\tilde{f}(\{\alpha e_{x_0,x}\})=&\{e_{x_0,f(x_0)}^0 f_*(\alpha e_{x_0,x})\}\nonumber \\=&\{e_{x_0,f(x_0)}^0 f_*(\alpha) f_*( e_{x_0,x})\}.
\end{align}
Recall that $\tilde{f}(\tilde{x})=\{e_{x_0,f(x_0)}^0 f_*( e_{x_0,x})\}
$.
Then
\begin{align}\label{eq3.07'a}
(\tilde{f})_{\#}(\alpha)
= & e_{x_0,f(x_0)}^0\,f_*(\alpha e_{x_0,x})
\big(e_{x_0,f(x_0)}^0\,f_*( e_{x_0,x})\big)^{-1}\nonumber \\=& e_{x_0,f(x_0)}^0\,f_*(\alpha)(e_{x_0,f(x_0)}^0)^{-1} \,.
\end{align}
Equality \eqref{eq3.07f} implies by induction
\begin{align}\label{eq3.07'''}
\tilde{f}^n \circ\alpha^{cov}(\tilde{x})=\big((\tilde{ f})_{\#}^n(\alpha)\big)^{cov}(\tilde{f}^n(\tilde{x})),\,  \tilde{x}\in \tilde{ X}\, .
\end{align}
Indeed,
\begin{align*}
\tilde{f}\big(\tilde{f}^{n-1} \circ\alpha^{cov}(\tilde{x})\big)=\tilde{f}\Big(\big((\tilde{ f})_{\#}^{n-1}(\alpha)\big)^{cov}\big(\tilde{f}^{n-1}(\tilde{x})\big)\Big)=\Big((\tilde{f})_{\#}^n(\alpha)\Big)^{cov}(\tilde{f}^n(\tilde{x}))\,.
\end{align*}

\begin{lemm}\label{lem3.5}
If $f_0$ and $f_1$ are two self-homeomorphisms of $X$ that are isotopic to each other then $(\tilde{f}_0)_{\#}= (\tilde{f}_1)_{\#}$.
\end{lemm}

\noindent {\bf Proof}
Let $f_t$ be a homotopy of self-homeomorphisms of $X$ joining $f_0$ and $f_1$.
Denote by $f:[0,1] \times X \to  [0,1] \times X$ the mapping for which $f(t,x)= f_t(x),\, t \in [0,1],\, x \in X$. The inverse of $f$ is the mapping defined by $f^{-1}(t,x)= (f_t)^{-1}(x),\, t \in [0,1], x \in X$.  Denote by $\tilde f$ a lift of $f$ to $[0,1] \times \tilde X$. 
Fix a covering transformation $\alpha \in \pi_1(X,x_0)$. Let $\tilde X_0$ be a compact subset of $\tilde X$ whose interior covers $X$.

The mapping $\tilde f$ and its inverse ${\tilde f}^{-1}$ are uniformly continuous on $[0,1] \times (\tilde X_0 \cup \alpha(\tilde X_0))$. Hence, for each  $\varepsilon >0$  there exists $\delta>0$ such that the implication
\begin{align}\label{eq3.100'}
(t_1,\tilde{x}_1),(t_2,\tilde{x}_2) \in [0,1]\times &  (\tilde X_0 \cup \alpha(\tilde X_0))    ,\nonumber\\
|t_1-t_2|<\delta,\; d(\tilde{x}_1,\tilde{x}_2)<& \,\delta \; \Rightarrow\;
d(\tilde{f}_{t_1}(\tilde{x}_1),\tilde{f}_{t_2}(\tilde{x}_2))< \varepsilon
\end{align}
holds, 
and the respective implication holds for the inverse of $\tilde f$.

Put $\xi(\tilde{x})=\tilde{f}_{t_1}(\tilde{x})$. We want to show that
\begin{equation}\label{eq3.103}
d\Big(\big(({\tilde{f}_{t_2}})_{\#}(\alpha)\big)^{cov}(\xi(\tilde{x})),\big( ({\tilde{f}_{t_1}})_{\#}(\alpha)\big)^{cov}(\xi(\tilde{x})\Big) < 2\varepsilon
\end{equation}
for $\tilde{x} \in  \tilde X_0 $ and $|t_1-t_2|<  \delta$.
Equation \eqref{eq3.103} shows that the values of the covering transformations $\big((\tilde{f_{t_1}})_{\#}(\alpha)\big)^{cov}$ and
$\big((\tilde{f_{t_2}})_{\#}(\alpha)\big)^{cov}$ are $2\varepsilon$-close on the set $\tilde{f_{t_1}}(\tilde X_0 )$. Since $\tilde{f_{t_1}}$ is a lift of a self-homeomorphism of $X$ this set covers $X$. If $\varepsilon$ is small, inequality \eqref{eq3.103} can only hold, if the covering transformations coincide. This implies  that $\big((\tilde{f_{t}})_{\#}(\alpha)\big)^{cov}$ is locally constant for $t \in [0,1]$, hence, it is constant and $\big((\tilde{f_{0}})_{\#}(\alpha)\big)^{cov}=\big((\tilde{f_{1}})_{\#}(\alpha)\big)^{cov}$. Since $\alpha$ was an arbitrary covering transformation the lemma will be proved if inequality \eqref{eq3.103} is proved.

We prove now inequality \eqref{eq3.103}. By the implication \eqref{eq3.100'}
the inequality $d(\tilde{f}_{t_1}\circ\alpha^{cov}(\tilde{x}),\tilde{f}_{t_2}\circ\alpha^{cov}(\tilde{x}))<\varepsilon$ holds for $\tilde{ x}\in \tilde{X}_0$ and $t_1,t_2\in [0,1],\,|t_1-t_2|<\delta$.
Hence, by \eqref{eq3.07f}
\begin{align}\label{eq3.105}
d\Big( \big((\tilde{f}_{t_1})_{\#}(\alpha)\big)^{cov}(\tilde{f}_{t_1}(\tilde{x})), \big((\tilde{f}_{t_2})_{\#}(\alpha)\big)^{cov}(\tilde{f}_{t_2}(\tilde{x}))\Big)<\varepsilon
\end{align}
for those $\tilde{x},\;t_1$ and $t_2$. Since the covering transformation $\big(\tilde{f}_{t_2})_{\#}(\alpha)\big)^{cov}$ is an isometry, the implication \eqref{eq3.100'} implies
the inequality
\begin{align}\label{eq3.106}
d\Big(\big((\tilde{f_{t_2}})_{\#}(\alpha)\big)^{cov}(\tilde{f_{t_2}}(\tilde{x})),\big((\tilde{f_{t_2}})_{\#}(\alpha)\big)^{cov}(\tilde{f_{t_1}}(\tilde{x}))\Big)<\varepsilon\,.
\end{align}
For $\tilde{x}\in\tilde{X}_0,\, |t_1-t_2|<\delta$.
Inequalites \eqref{eq3.105} and \eqref{eq3.106} imply inequality \eqref{eq3.103}. The lemma is proved. \hfill $\Box$

\medskip

\medskip

For two lifts $\tilde{f}_1$ and $\tilde{f}_2$ of $f$ the associated isomorphisms  $ ({\tilde{f}_1)}_{\#}$ and  $({\tilde{f}_2)}_{\#}$ of $\pi_1(X,x_0)$ are conjugate. Indeed, for a covering transformation $\theta$
\begin{equation}\label{eq3.107}
\tilde {f_1}\circ \alpha^{cov}(\tilde{x}) = \theta \circ\tilde{f_2}\circ\alpha^{cov} (\tilde{x}) = \theta\circ ({\tilde{f}_2})_{\#}(\alpha) (\tilde f_2(\tilde{x})) = \theta\circ ({\tilde{f}_2)}_{\#}(\alpha)(\theta^{-1} \tilde {f_1}(\tilde{x}))\,. \nonumber
\end{equation}
We obtain the equality $({\tilde{f}_1})_{\#}(\alpha)=\theta ({\tilde{f}_2})_{\#}(\alpha)\theta^{-1}$.
In other words, the mapping $f$ defines a conjugacy class of isomorphisms of $\pi_1(X,x_0)$, denoted by $\widehat{f_{\#}}$.
By Proposition \ref{prop3.101'}  the quantities  $\Gamma_{{({\tilde{f_j}})_{\#}}}$  defined by equation \eqref{eq3.101}  for the isomorphisms $ ({\tilde{f}_j})_{\#}   $ of $\pi_1(X,x_0)$ satisfy the equality $\Gamma_{{({\tilde{f_1}})_{\#}}}=\Gamma_{{ ({\tilde{f}_2})_{\#}}}$. Hence, there is a
quantity $\Gamma_{\widehat{f_{\#}}}$ related to $f$, that is correctly defined by
putting it equal to $\Gamma_{{ {\tilde f}_{\#}}}$ for any lift $\tilde f$ of $f$. Moreover, by Lemma
\ref{lem3.5}  $\Gamma_{\widehat {f_{\#}}}$ depends only on the isotopy class of $f$.

\medskip

The following lemma relates the word length $\mathcal{L}_{\pi_1(X,x_0)}(g)$ of an element $g \in \pi_1(X,x_0)$ to the distance from a point in the universal covering $\tilde X$ to its image under $g$. The proof of this lemma is due to J.Milnor \cite{Mi}. (See also \cite{FLP}, Lemma 10.7 there.) Again, $d$ is a metric on $X$ and its lift to the universal covering $\tilde X$ is denoted by the same letter $d$. Notice that the metric $d$ on $\tilde X$ is invariant under covering transformations.

\begin{lemm}\label{lem3.107}
Fix a point $\tilde {x} \in \tilde X$ and a set $\mathcal{G}$ of generators of the fundamental group $\pi_1(X,x_0)$. There exist two positive constants $C_1$ and $C_2$
such that for each $g \in \pi_1(X,x_0)$ the inequality
\begin{equation}\label{eq3.110}
C_1 \mathcal{L}_{\mathcal{G}}(g) \leq d(\tilde{x},g(\tilde{x})) \leq C_2 \mathcal{L}_{\mathcal{G}}(g) \,
\end{equation}
holds.
\end{lemm}
\noindent {\bf Proof.}  Let $\delta$ be the diameter of $X$. Put $\tilde {N}=\{\tilde{y}\in \tilde{ X}: d(\tilde{y},\tilde{x})\leq \delta\}\,.$ Then for the projection ${\sf P}:\tilde{X}\to X$ we have ${\sf P}(N)=X$.
The sets $\{g(\tilde{N}):g \in \pi_1(X,x_0)\}$ cover $\tilde{X}$. Indeed, for each point $\tilde y$ in $\tilde X$
there is a point $\tilde{y}'\in \tilde N$ with $p(\tilde{y}')= p(\tilde{y})=y$. There exists a covering transformation $g$ such that $g(\tilde{y}')=(\tilde{y})$.

The family $g(\tilde{N}),\,g\in\pi_1(X,x_0)$, is a locally finite cover of $\tilde X$ (by compact subsets). Indeed, let $\tilde{U}(\tilde x)\subset \tilde N$ be a neighbourhood of $\tilde{x}$ in $\tilde X$ such that $g(\tilde{U}(\tilde x))\cap \tilde{U}(\tilde x)=\emptyset$ for each $g\in \pi_1(X,x)$. If $g(\tilde{N})$ intersects $\tilde N$ then, since $g$ is an isometry, $g(\tilde{N})$ is contained in $\{\tilde{y}\in \tilde{ X}: d(\tilde{y},\tilde{x})\leq 3\delta\}$. If there were infinitely many different $g_j\in \pi_1(X,x_0)$ for which $g_j(\tilde{N})\cap \tilde{N}\neq \emptyset $, then since all
$g_j(\tilde{U}(\tilde x))$ are disjoint, have the same non-zero area in the metric $d$ and are  contained in $\{\tilde{x}\in \tilde{ X}: d(\tilde{y},\tilde{x})\leq 3\delta\}$, the area of the set
 $\{\tilde{y}\in \tilde{ X}: d(\tilde{y},\tilde{x})\leq 3\delta\}$ would be infinite. This is impossible.

We saw that the set $$\mathcal{G}'\stackrel{def}=\{ g \in \pi_1(X,x_0): g\tilde{N}\cap \tilde{N} \neq \emptyset\}$$ is finite. Since the sets $g(\tilde{N}),\,g\in\pi_1(X,x_0)$,
cover $\tilde X$, $\mathcal{G}'$
generates $\pi_1(X,x_0)$. Notice that $\mathcal{G}'$ contains the inverse of each of its elements.

Moreover, there is a positive number $\nu$ such that $d(\tilde{N}, g(\tilde{N}))\geq \nu$ if $g(\tilde{N})\cap \tilde{N}=\emptyset$.
Indeed, the union of all $g'(\tilde{N})$, $g'\in\pi_1(X,x_0)$, that intersect $\tilde N$, contains an open neighbourhood of the closure $\overline{\tilde N}$. If for some $g\in \pi_1(X,x_0)$ the set $g(\tilde{N})$ does not intersect this open neighbourhood of $\overline{\tilde N}$, the distance of $g(\tilde{N})$ to $\tilde N$ is bigger than a positive constant not depending on $g$. There are only finitely many $g$ for which $g(\tilde{ N})$ intersects the open neighbourhood of $\tilde N$, and each $g(\tilde{N})$ has positive distance to $\tilde N$, if it does not intersect $\tilde N$.

The second inequality in \eqref{eq3.110} is now easy to prove. If for an element $g\in\pi_1(X,x_0)$ we have $\mathcal{L}_{\mathcal{G}}(g)=n$, then  $\mathcal{L}_{\mathcal{G}'}(g)=n'$ for $n'\leq (\max \mathcal{L}_{\mathcal{G}'}(g_j))\,\mathcal{L}_{\mathcal{G}}(g)= C'_2 n$ with $C'_2=\max \mathcal{L}_{\mathcal{G}'}(g_j)$.  We can write $g=g_1\,g_2\,\ldots g_{n'}$ with $g_j(\tilde{N})\cap \tilde{N}\neq\emptyset$. Then $d(\tilde{x}, g( \tilde{x}))\leq 2\delta n'\leq 2\delta C_2'\,n =2\delta C_2'\, \mathcal{L}_{\mathcal{G}}(g)$.

The first inequality in  \eqref{eq3.110} is obtained as follows. Let $g\in \pi_1(X,x_0)$. Take the smallest positive integer number $k$ for which $d(\tilde{x}, g(\tilde{x}))< k\nu $. Consider a sequence of points $\tilde{y}_0=\tilde{x},\ldots, \tilde{y}_{k-1},\tilde{y}_k= g(\tilde{x})$, such that $d(\tilde{y}_j,\tilde{y}_{j+1})<\nu$ for $j=0,1,\ldots,k-1$. Choose for $j=1,\ldots,k-1$ a point $\tilde{y}'_j\in \tilde N$ and an element $g_j\in\mathcal{G}'$ such that $\tilde{y}_{j}=g_j(\tilde{y}'_{j})$ and put $g_0$ equal to the identity and $g_k=g$. 
We obtain $d(g_j(\tilde{y}'_{j}),g_{j+1}(\tilde{y}'_{j+1}))<\nu$. Hence, $g_j^{-1}g_{j+1}\in \mathcal{G}'$. Since $g=(g_0^{-1}g_1)\ldots(g_{k-1}^{-1}g_k)$, we obtain $L_{\mathcal{G}'}(g)< k$.

Since $k$ is minimal, we have with $\mu= \min\{d(\tilde{x},g(\tilde{x})): g\neq{\rm Id}, g \in \pi_1(X,x_0)\}$
$$
L_{\mathcal{G}'}(g)\leq \frac{1}{\nu}d(\tilde{x},g(\tilde{x})) +1\leq (\frac{1}{\nu}+\frac{1}{\mu})d(\tilde{x},g(\tilde{x})).
$$
The lemma follows by the inequality $L_{\mathcal{G}}(g) \leq (\max L_{\mathcal{G}}(g_j'))  L_{\mathcal{G}'}(g)$. \hfill $\Box$

\medskip

We give now the proof of the following theorem from expos\`{e} 10 of \cite{FLP} which is interesting in itself. It gives a lower bound of the entropy by the quantity $\Gamma_{\widehat{f_{\#}}}$.

\begin{thm}\label{thm3.101}
For a self-homeomorphism $f$ of $X$ the inequality
\begin{equation}\label{eq3.111}
h(f) \geq \Gamma_{\widehat{f_{\#}}}
\end{equation}
holds.
\end{thm}

\noindent {\bf Proof}. Let $\tilde f$ be a lift of $f$ to the universal covering $\tilde X$. By
Milnor's Lemma with $g$ replaced by  $\tilde{f}_{\#}^n(g)\in \pi_1(X,x_0)$ for each $\tilde{x} \in \tilde X$
\begin{align}\label{eq3.112}
\Gamma_{{\tilde{f}_{\#}}}= & \underset {g \in \pi_1(X,x_0)} \sup \limsup \frac{1}{n}\log\mathcal{L}_{\mathcal{G}}(\tilde{f}_{\#}^n(g))\nonumber \\
= & \underset {g \in \pi_1(X,x_0)} \sup \limsup \frac{1}{n}
\log d(\tilde{x}, (\tilde{f}_{\#})^n(g)(\tilde{x}))
\end{align}
for a metric $d$ on $X$ and its lift to $\tilde X$ denoted also by $d$.
Take in Lemma \ref{lem3.106}
$\tilde{y}=g(\tilde{x})$. Since by equation \eqref{eq3.07'''} $\tilde f ^n(g(\tilde{x})) = (\tilde{f}_{\#})^n(g)(\tilde f^n (\tilde{x}))$, Lemma \ref{lem3.106}
gives
\begin{equation}\label{eq3.113}
\limsup \frac{1}{n} d(\tilde f ^n(\tilde{x}),  (\tilde{ f}_{\#})^n(g)(\tilde f ^n(\tilde{x})))  \leq h(f) .
\end{equation}
Note that
\begin{align}\label{eq3.114'}
d(\tilde{x}, (\tilde{f}_{\#})^n(g)(\tilde{x})) & \leq d(\tilde{x}, \tilde f ^n(\tilde{x}))+d(\tilde f ^n(\tilde{x}),  (\tilde{ f}_{\#})^n(g)(\tilde f ^n(\tilde{x})))\nonumber \\
& + d((\tilde{ f}_{\#})^n(g)(\tilde f ^n(\tilde{x})), (\tilde{ f}_{\#})^n(g)(\tilde{x}))\,.
\end{align}
The first term on the right is estimated by
\begin{equation}\label{eq3.114}
d(\tilde{x}, \tilde f ^n(\tilde{x})) \leq \sum_{\ell =0}^{n-1} d( \tilde f^{\ell}(\tilde{x}), \tilde f^{\ell}(\tilde f(\tilde{x})) ).
\end{equation}
Apply Lemma \ref{lem3.106} with
$\tilde{y}= \tilde f(\tilde{x})$, and Lemma \ref{lem3.100}, Statement (2).  We obtain from \eqref{eq3.114}
\begin{equation}\label{eq3.115}
\limsup \frac{1}{n} \log d(\tilde{x},\tilde f ^n(\tilde{x})) \leq h(f).
\end{equation}
The last quantity $d((\tilde{f}_{\#})^n(g)(\tilde{x}), (\tilde{ f}_{\#})^n(g)(\tilde f ^n(\tilde{x})))$ on the right of \eqref{eq3.114'} is equal to the left hand side of equation \eqref{eq3.114}
since the covering transformation $(\tilde{f}_{\#})^n(g)$ is an isometry in the metric $d$. Hence, $d(\tilde{x}, (\tilde{f}_{\#})^n(g)(\tilde{x}))$ can be estimated from above by the sum of the three positive terms $I_j(n)$, $j=1,2,3,$ on the right of \eqref{eq3.114'} with $\limsup \frac{1}{n} \log I_j(n) \leq h(f)$ for each $j$. By Lemma \ref{lem3.100}, Statement (1), inequality \eqref{eq3.111} follows. \hfill $\Box$

\medskip
\index{Poincar\'{e}-Hopf  ! Theorem}
We need the Poincar\'{e}-Hopf Theorem. We consider a compact differentiable manifold $X$ of dimension $n$ with or without boundary and a smooth tangent vector field $v$ on $X$. If $X$ has a boundary then $v$ is required to point out of $X$ at points of $\partial X$. The singularities of $v$ are its zeros on $X$. We assume
that the zeros are isolated and contained in the interior of $X$. The index $i_v(p)$ of $v$ at an isolated singularity $p$ is defined as follows. Consider local coordinates in a small neighbourhood of $p$.
Then the mapping $z \to \frac{v(z)}{|v(z)|}$ takes a small sphere around $p$ in these coordinates to the unit sphere in $\mathbb{R}^{n-1}$. The degree of this mapping, considered as mapping from the unit sphere to itself, is called the index if $v$ at $p$. The index of any regular (i.e. non-singular) point equals zero.

\medskip

\noindent {\bf Poincar\'{e}-Hopf Theorem.} {\it Let $X$ be a compact differentiable manifold (or manifold with boundary) of dimension $n$, and $v$ a smooth tangent vector field on $X$ with isolated zeros. If $X$ has a boundary then we require that $v$ has no singularities on $\partial X$, and $v$ points out of $X$ at points of $\partial X$. Then
\begin{equation}\label{eq3.118}
\sum i_v(p) = \chi(X)\,,
\end{equation}
where $\chi(X)$ is the Euler characteristic of $X$.}
\index{Euler ! characteristic}
\bigskip

For a proof see \cite{Mi2} or \cite{Br}.

The Poincar\'{e}-Hopf Theorem implies the following theorem (see Proposition 5.6 in \cite{FLP}). Actually, we need only the version of the  Poincar\'{e}-Hopf Theorem for the dimension $2$ which was proved already by Poincar\'{e}.

\begin{thm}\label{thm3.102} Let $X$ be a Riemann surface, and let $\phi$
be a holomorphic quadratic differential on $X$ (equivalently, a meromorphic quadratic differential without poles). Then there is no relatively compact topological disc in $X$ with piecewise smooth boundary consisting of the union of a (perhaps empty) closed smooth vertical arc and a (perhaps empty) closed smooth arc that is contained in the regular part of $X$ and is transversal to the
vertical foliation.
\end{thm}

\noindent {\bf Proof.} Suppose, in contrary, that such a disc $\Delta$ exists. \index{$\Delta$}
Then by Proposition \ref{prop3.102} the whole boundary $\partial \Delta$ cannot be a vertical curve. If the whole boundary $\partial \Delta$ is transversal to the vertical foliation, we may deform the boundary of $\Delta$ slightly within the regular part of $X$, so that the boundary of the new disc is piecewise smooth and equals the union of a non-empty smooth vertical arc and a non-empty smooth arc that is transversal to the vertical foliation. We will assume from the beginning that $\partial \Delta$ is piecewise smooth and equals the union of a non-empty smooth closed vertical arc $\gamma_v$ and a non-empty smooth closed arc $\gamma_t$ that is contained in the $\phi$-regular part and is transversal to the vertical foliation.

Consider a conformal mapping of $\Delta$ onto the unit half-disc $\mathbb{D}_r= \{z\in \mathbb{C}_r: |z|<1\}$ in the right half-plane $\mathbb{C}_r$, whose continuous extension to the boundary maps the vertical arc $\gamma_v$ onto the interval $[-i,i]$. Since the vertical segment $\gamma_v$ is real analytic, the conformal mapping extends holomorphically by the reflection principle across the interval  $(-i,i)$. The extension maps a neighbourhood $N$ of the arc $\gamma_v$ to a neighbourhood $V$ of the segment $(-i,i)$.

Consider the push-foreword $\phi'$ to $\mathbb{D}_r\cup V$ of the quadratic differential $\phi\mid (\Delta\cup N)$ under the conformal mapping. 
For each singular point of $\phi'$ that is contained in ${\mathbb{D}_r}$
we consider a closed arc that joins the singular point with a point in the boundary half-circle, such that all points of the arc except one endpoint are contained in  ${\mathbb{D}_r}$. We may choose the arcs pairwise disjoint. 
Let $\Omega'$ be the complement of the arcs in  ${\mathbb{D}_r}$, and let
$\Omega$ be the union of  $\Omega'$ 
with its reflection in the imaginary axis and with the interval $(-i,i)$.

After possibly shrinking $V$, the domain  $V\cup\Omega'$ 
is simply connected, and the quadratic differential $\phi'\mid V\cup\Omega'$ has the form
$\phi'(\zeta) d\zeta^2$ for a non-vanishing holomorphic function $\phi'$ on $V\cup\Omega'$. Consider the form $\sqrt{\phi'}d\zeta$ on $\Omega'\cup V$ for a branch of the square root of $\phi'$. Since the segment $(-i,i)$ is a vertical trajectory of the 
quadratic differential $\phi'$, the real part of the vertical vector field $v_v$ vanishes at points of $(-i,i)$. With $\sqrt{\phi'(\zeta)}=\sqrt{|\phi'(\zeta)|}e^{\frac{i\theta}{2}}$
and $v_{v}(\zeta)=|\phi(\zeta)|^{-\frac{1}{2}}\binom{\sin \frac{\theta}{2}}{ \cos\frac{\theta}{2}}$ 
we obtain ${\rm Im}(\sqrt{\phi'})\mid (-i,i)=0$ (see equality \eqref{eq2.300} and the following calculations).

By Schwarz Reflection Principle the function $\sqrt{\phi'}\mid \Omega'$ extends to a holomorphic function on $\Omega$. Hence $\phi' \mid \Omega'$ extends to a holomorphic function on $\Omega$. 
Since
$\phi' d\zeta^2\mid \Omega'$  extends across the cuts to a holomorphic quadratic differential on $\mathbb{D}_r$ whose sigular points are  the initial points of the cuts, the reflection provides a holomorphic quadratic differential on $\mathbb{D}$, denoted by   $\phi_{\mathbb{D}}$,
whose singular points are the singular points of the push-foreward of $\phi\mid \Delta$ and their reflections in the imaginary axis.
\index{$E_{even}$} \index{$E_{odd}$}

Denote by $E_{even}$ the set of singular points of  $ \phi_{\mathbb{D}}$ of even order and by $E_{odd}$ the set of singular points of odd order. We want to apply the Poincar\'{e}-Hopf Theorem. If the set $E_{odd}$ is not empty,  $ \phi_{\mathbb{D}}$ does not define a vector field on $\mathbb{D}$. We consider the double branched covering $S$ of $\mathbb{D}$ with branch locus $E_{odd}$. Recall that the points of $E_{odd}$ come in pairs. Let $\tilde\phi_{\mathbb{D}}$ be a lift of  $ \phi_{\mathbb{D}}$ to $S$. In distinguished coordinates $z$ at any point in  $E_{even}$ the quadratic differential  $ \phi_{\mathbb{D}}$ is written as
$(\frac{a+2}{2})^2 z^a (dz)^2$ with a positive even number $a$, and the lift of  $ \phi_{\mathbb{D}}$ has the same behavior near the lifts of points in $E_{even}$.

At any point in  $E_{odd}$  the quadratic differential  $ \phi_{\mathbb{D}}$ is written
in distinguished coordinates $z$  as
$(\frac{a+2}{2})^2 z^a (dz)^2$ with a positive odd number $a$. The lift $\tilde \phi_{\mathbb{D}}$ can be written in coordinates $\zeta$ on $S$ with
$z(\zeta)= \zeta^2$ as $\tilde\phi_{\mathbb{D}}(\zeta)= \phi_{\mathbb{D}}(z(\zeta)) (z'(\zeta))^2= (a+2)^2 \zeta^{2a+2}$.

The quadratic differential $\tilde\phi_{\mathbb{D}}$ defines a vector field on $S$.
Indeed, cut $\mathbb{D}$ along disjoint arcs that join points $E_{odd}$ with the circle $\partial \mathbb{D}$. We obtain a simply connected domain $\hat \Omega$
which is contained in the unit disc and contains only singular points of $\phi_{\mathbb{D}}$ of even order. Hence, each point in the obtained domain $\hat\Omega$ has a neighbourhood on which two branches of the differential $\sqrt{\phi_{\mathbb{D}}(z)}dz$ are defined. The two  branches differ by sign. Since the domain is simply connected there are two branches of $\sqrt{\phi_{\mathbb{D}}(z)}dz$ that are globally defined on $\hat\Omega$. The surface $S$ is obtained as follows. Take two copies of $\hat\Omega$ which we call sheets. Consider for each cut in the disc the respective copies of the two sheets and do cross-gluing of the edges. Take a global branch $\sqrt{\phi_{\mathbb{D}}(z)}dz$ on one sheet and $-\sqrt{\phi_{\mathbb{D}}(z)}dz$ on the other sheet. The behaviour along small closed curves surrounding any preimage of a point in $E_{odd}$ shows that we obtain a globally defined differential which we denote by $\sqrt{\phi_{\mathbb{D}}(z)}dz$. Consider the vector field $v(z)$ on $S$ that is defined by $\sqrt{\phi_{\mathbb{D}}(z)}dz(v(z))=i$ for the non-singular points $z$ (a vector field in the vertical direction). The vector field $v$
is transversal to the boundary. Changing $v$ to $-v$ if necessary we may assume that $v$ points out of $S$ at all boundary points.

Apply the Poincar\'{e}-Hopf Theorem.
Take a point in $S$ whose projection to $\mathbb{D}$ is in $E_{ even}$.
In distinguished coordinates the quadratic differential $\tilde\phi_1$ near this point has the form $(\frac{a+2}{2})^2 z^{a} (dz)^2$ with $a$ even.
The vertical vector field near this point is given by $v(z)=\frac{2}{a+2} i z^{-\frac{a}{2}}$.
The index of the vector field $v$ at this point equals $-\frac{a}{2}$.
For a point in $S$ whose projection to $\mathbb{D}$ is in $E_{odd}$. the quadratic differential equals $(a+2)^2 \zeta^{2a+2}d\zeta^2$ and the vertical vector field can be written as $v(\zeta)= \frac{1}{a+2}i z^{-(a+1)}$. The index equals $-(a+1)$.
For a point $p \in \mathbb{D}$ we denote by $a(p)$ the order of $\phi_1$ at this point. Then
\begin{equation}\label{eq3.119}
\sum_{p \in S} i(p) = 2 \sum _{p \in E_{even}}-\frac{a(p)}{2} + 2 \sum _{p \in E_{odd}}-(1+a(p))\,.
\end{equation}
For the Euler characteristic of the double branched covering $S$ we have
$$
\chi(S)= 2 \chi(\mathbb{D}) - B= 2 -B \,.
$$
Here $B$ is the number of branch points which is equal to $2\,|E_{odd}|$, where  $|E_{odd}|$ is the number of points in $E_{odd}$. Hence, by the Poincar\'{e}-Hopf Theorem
\begin{equation}\label{eq3.120}
- \sum _{p \in E_{even}} a(p) - 2 \sum _{p \in E_{odd}} a(p)- 2 \,|E_{odd}|
= 2 - 2 \,|E_{odd}| \,.
\end{equation}
This is impossible. The contradiction proves the theorem. \hfill $\Box$

\bigskip
\index{Douadi Lemma}
The following lemma is proved by Douady (see Lemma 9.22 in \cite{FLP}).
\begin{lemm}\label{lem3.108}
Let $X$ be a compact Riemann surface with a quadratic differential $\phi$ and let
$\varrho$ be any Riemannian metric on $X$. There exist positive constants $c_1$ and $c_2$ such that for any free homotopy class $\hat\alpha$ of closed curves on $X$ the inequality
\begin{equation}\label{eq3.121}
c_1\,\ell_{\phi}(\hat \alpha) \leq \ell_{\varrho}(\hat \alpha)\leq  c_2\,\ell_{\phi}(\hat \alpha)\,
\end{equation}
holds.
Here $\ell_{\phi}(\hat \alpha)$ is the infimum of the $\phi$-length over all piecewise smooth loops in $\hat \alpha$, and $\ell_{\varrho}(\hat \alpha)$ is the infimum of the length 
in the metric $\varrho$ over all piecewise smooth loops 
in $\hat \alpha$.
\end{lemm}
\index{$\ell_{\phi}(\hat \alpha)$}   \index{$\ell_{\varrho}(\hat \alpha)$}
\medskip

\noindent {\bf Proof.} Take a small enough number $\delta>0$, such that the
$\delta$-neighbourhood $V_{\delta}(z_j)$ in the $\varrho$-metric of any singular point $z_j$ of $\phi$ is simply connected.
Any loop that is contained in such a neighbourhood $V_{\delta}(z_j)$ is contractible to the singular point $z_j$. 
Hence, if $\hat \alpha$ contains a loop contained in $V_{\delta}(z_j)$, than both, $\ell_{\phi}(\hat \alpha)$ and $\ell_{\varrho}(\hat \alpha)$ are equal to zero. In this case the inequalities \eqref{eq3.121} hold.

On the other hand, there exist positive numbers $c_1'$ and $c_2'$ depending on $\delta$ such that for any curve $\gamma'$ avoiding the $\frac{\delta}{2}$-neighbourhoods in the $\varrho$-metric of the singular points the inequalities
\begin{equation}\label{eq3.122a}
c_1'\,\ell_{\phi}( \gamma') \leq \ell_{\varrho}( \gamma')\leq  c_2'\,\ell_{\phi}( \gamma')\,
\end{equation}
hold. Moreover, there exists a constant $c''\geq 1$ such that for each $j$
\begin{align}\label{eq3.122b}
\ell_{\varrho}(\partial V_{\delta}(z_j)) \leq & 2 c'' d _{\varrho}(\partial V_{\delta}(z_j)), \partial V_{\frac{\delta}{2}}(z_j))\,,\nonumber\\
\ell_{\phi}(\partial V_{\delta}(z_j)) \leq & 2 c'' d _{\phi}(\partial V_{\delta}(z_j)), \partial V_{\frac{\delta}{2}}(z_j))\,.
\end{align}

Suppose the free homotopy class of curves $\hat\alpha$ cannot be represented by a loop contained in one of the $V_{\delta}(z_j)$. Take a curve $\gamma \in \hat \alpha$ which almost minimizes the $\varrho$-length,  $\ell_{\varrho}(\gamma)\leq \ell_{\varrho}(\hat \alpha)+ \varepsilon$ for a small positive number $\varepsilon$. For the part $\gamma \setminus \cup V_{\frac{\delta}{2}}(z_j)$ outside the $\frac{\delta}{2}$-neighbourhoods of the critical points the estimates \eqref{eq3.122a} hold.

Suppose there is a $j$ and a connected component $\tilde\gamma$ of $\gamma \cap V_{\delta}(z_j)$ which intersects  $V_{\frac{\delta}{2}}(z_j)$.
The connected component $\tilde \gamma$ is an arc with the two endpoints
on $\partial V_{\delta}(z_j)$. Replace $\tilde \gamma$ by an arc $\tilde \gamma'$ on  $\partial V_{\delta}(z_j)$ which is homotopic to $\tilde \gamma$ with fixed endpoints. By \eqref{eq3.122b}
\begin{align}\label{eq122c}
\ell_{\varrho}(\tilde \gamma') \leq \ell_{\varrho}(\partial V_{\delta}(z_j))\leq 2 c'' d _{\varrho}(\partial V_{\delta}(z_j)), \partial V_{\frac{\delta}{2}}(z_j)) \leq c'' \ell_{\varrho}(\tilde \gamma)\,.
\end{align}
Consider for each singular point $z_j$ all connected components of the intersection  $\gamma \cap V_{\delta}(z_j)$ which have points in  $V_{\frac{\delta}{2}}(z_j)$.
Replace in the same way as above each such component by an arc 
which avoids the sets $V_{\frac{\delta}{2}}(z_j)$. We obtain a curve $\gamma'$ that is homotopic to $\gamma$, avoids the ${\frac{\delta}{2}}$-neighbourhoods in the $\varrho$-metric of the singular points,  and satisfies the estimates
\begin{align}\label{eq122d}
\ell_{\varrho}( \gamma') \leq c'' \ell_{\varrho}(\gamma) 
\leq {c''} (\ell_{\varrho}(\hat \alpha) + \epsilon) \,.
\end{align}
Hence, by the inequalities \eqref{eq3.122a} the inequality $\ell_{\phi}(\hat \alpha) \leq \frac{c''}{c_1'}( \ell_{\varrho}(\hat \alpha) + \varepsilon)$ holds. Since $\varepsilon$ can be taken to be an arbitrary positive number, the first inequality in \eqref{eq3.121} holds. The second inequality follows by interchanging the role of $\varrho$ and $\phi$. 
\hfill $\Box$

\bigskip

The following proposition is proved in \cite{FLP} (see Proposition 5.7. there).

\begin{prop}\label{prop3.102'}
Let $X$ be a closed Riemann surface, and let $\phi$ be a holomorphic quadratic differential on $X$. Suppose $\gamma$ is a simple closed (connected and smooth) curve in $X$ which is contained in the $\phi$-regular part of $X$ and is transversal to the vertical foliation.
Let $\gamma'$ be any closed (connected and piecewise smooth) curve in $X$ which is free homotopic to $\gamma$. Then the following estimate for the horizontal $\phi$-length holds
\begin{equation}\label{eq3.124}
\ell_{\phi,h}(\gamma) \leq \ell_{\phi,h}(\gamma') \,.
\end{equation}
\end{prop}

\noindent {\bf Proof.} Let $x_0$ be the base point of $\gamma$. Denote by $\alpha$ the class of $\gamma$ in the fundamental group $\pi_1(X,x_0)$, and let ${\sf P}:\tilde{ X}\to X$ be the unversal covering of $X$. Consider the quotient $\hat{X}\stackrel{def}=X\diagup \alpha^{cov}$ for the covering transformation
$\alpha^{cov}$ associated to $\alpha$ and the fixed base point $\tilde{x}_0\in {\sf P}^{-1}(x_0)$. The quotient is conformally equivalent to an annulus. The universal covering  ${\sf P}:\tilde{ X}\to X$ induces a covering $\hat{p}:\hat{X}\to X$. (For more details see also Section \ref{sec:fin1a}). The holomorphis quadratic differential $\phi$ on $X$ lifts to a holomorphic
quadratic differential $\hat\phi$ on $\hat X$ and the free homotopic loops $\gamma$ and $\gamma'$ on $X$ lift to free homotopic loops $\hat\gamma$ and $\hat\gamma'$ on $\hat X$. Notice that $\ell_{\hat\phi,h}(\hat\gamma)=\ell_{\phi,h}(\gamma) $ and   $\ell_{\hat\phi,h}(\hat\gamma')=\ell_{\phi,h}(\gamma') $.  Hence, it is sufficient to prove the theorem in the case when $X$ is an annulus with a holomorphic quadratic differential, and both loops are free homotopic to the positive generator of the fundamental group of the annulus. We may assume that $\gamma'$ is also simple closed, removing otherwise some parts of $\gamma'$ which does not increase its horizontal $\phi$-length. We will now prove the Theorem under these conditions.

Assume first that $\gamma$ and $\gamma'$ are disjoint. Then the set $\gamma \cup\gamma'$ consists of all boundary points of an annulus $A\subset X$. Indeed, after a conformal mapping we may assume that $X$ is a subset of the plane. Then the claim follows from the Jordan Curve Theorem applied to each of the curves.

For each point $p \in \gamma$ we consider the vertical trajectory ray $r_p$ with initial point $p$ that enters $A$ at $p$. We call a trajectory ray $r_p$ exceptional if it stays in $A$ and runs into a critical point contained in $A$. We claim that
each non-exceptional vertical trajectory ray $r_p,\, p\in \gamma$,  meets the boundary $\partial A$ after $p$.
Indeed, if the non-exceptional ray $r_p$ is critical, it meets $\partial A$ at a point after $p$ by the definition of exceptional rays.

Suppose $r_p$ is non-critical. Consider an open $\phi$-rectangle $\mathcal{R}$ that contains $p\in \gamma$. Shrinking $\mathcal{R}$ in the horizontal direction if necessary, we may assume that the connected component $\gamma_p$, that contains $p$, of the intersection $\gamma\cap \mathcal{R}$  is an arc with endpoints on opposite vertical sides of $\mathcal{R}$. 

Denote by $\gamma^h_p$  the connected component containing $p$ of the intersection of $\mathcal{R}$ with the horizontal trajectory through $p$.
By (the proof of) Theorem \ref{thm3.03} the divergent vertical trajectory ray $r_p$ intersects $\gamma^h_p$ after $p$ infinitely often in the same direction.
Then the intersection
$r_p \cap \mathcal{R}$ contains a connected component that 
consists of points on $r_p$ after $p$. This component intersects $\gamma\subset \partial A$. 
The claim is proved.

By Proposition \ref{thm3.102} for any non-exceptional trajectory ray $r_p,\, p\in \gamma$, the first point after $p$ on the boundary $\partial A$ is contained in $\gamma'$. Indeed, suppose it is a point $p'$ on $\gamma$. The union of one of the arcs of $\gamma$ with endpoints $p$ and $p'$ and the arc of $r_p$ between $p$ and $p'$ bounds a disc. But the first arc is transversal to the vertical foliation and the second is an arc of a vertical leaf. This is  impossible.

We assign to each non-exceptional ray $r_p,\, p \in \gamma,$  the first point ${\sf{g}}(p)>p$ of the trajectory ray $r_p$  which is on $\gamma'$. Let $r_p'$ be the segment on $r_p$ between $p$ and ${\sf{g}}(p)$. 
There are only finitely many exceptional trajectory rays. Let $p_j,\, j=1,\ldots,N,$ be the initial points of the exceptional trajectory rays, ordered cyclically with respect to some orientation of $\gamma$. Let $s(p_j)$ be the singular point in $A$ which is the limit point of $r_{p_j}$. 
Then for $p_j\neq p_{j'}$ the singular points $s(p_j)$ and $s(p_{j'})$ are different. Indeed, suppose not. Then 
there are two neighbouring points $p_j$ and $p_{j+1}$
with $s(p_j)= s(p_{j+1})$. (If the two neighbouring points are $p_N$ and $p_1$ we relabel the $p_j$.) Then $r_{p_j}$ and $r_{p_{j+1}},$ are segments on neighbouring bisectrices at the singular point $s(p_j)=s(p_{j'})$, and the two segments together with an arc of $\gamma$ between $p_j$  and $p_{j+1}$ bound a sector which is laminated by vertical arcs with the two endpoints on the arc of $\gamma$ between $p_j$ and $p_{j+1}$. This is impossible by Theorem \ref{thm3.102}.
\begin{figure}[h]
\begin{center}
\includegraphics[width=80mm]{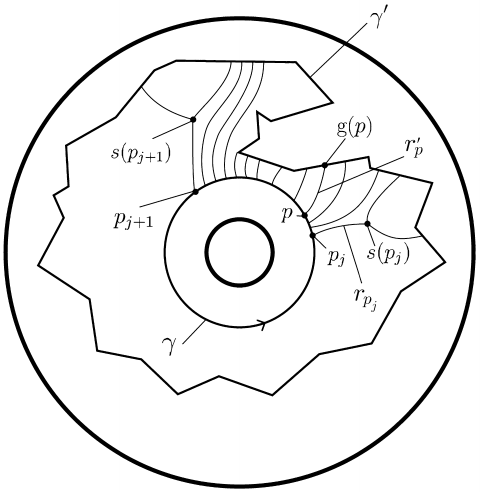}
\end{center}
\caption{Curves that are transversal to the vertical foliation are $\ell_{\phi,h}$-minimizing in their free homotopy class} \label{fig3.2}
\end{figure}

It is enough to prove the proposition for the case when $\gamma'$ does not contain singular points of the quadratic differential and is piecewise linear in flat coordinates. Indeed, one can achieve this by a deformation of $\gamma'$ that changes the $\phi$-length no more than by an a priory given amount.

Consider the open arc $\gamma_j$ contained in $\gamma$ whose endpoints are two neighbouring points
$p_j$ and $p_{j+1}$ ($p_N$ and $p_1$ for $j=N$).
We consider the set $\gamma'_j= {\sf{g}} (\gamma_j)$. The obtained subset $\gamma'_j$
of $\gamma'$ is the union of a finite number of open linear segments contained in $\gamma'$ and a finite number of points on $\gamma'$ (that are endpoints of some of the just mentioned linear segments). Hence, the set $\gamma'_j$ is measurable for the horizontal $\phi$-length. Moreover, each maximal open linear piece of  $\gamma'_j$ has the same horizontal $\phi$-length as the corresponding part on $\gamma_j$ which is the preimage of this linear piece under  $\sf{g}$. Indeed,
each point in the linear piece of $\gamma'$ is joined with its preimage under
$\sf{g}$ by a vertical segment.
Hence, $\ell_{\phi,h}(\gamma_j)=\ell_{\phi,h}(\gamma_j')$.
The union of the $\gamma_j$ equals $\gamma \setminus \cup\{p_j\}$.
The union of the $\gamma_j', \, j=1,\ldots N,$ may be much smaller than $\gamma'$. The proposition is proved for the case when $\gamma$ and $\gamma'$ are disjoint.

In the general case we may assume again that the curve $\gamma'$ does not contain singular points of the quadratic differential and is piecewise linear in flat coordinates. Consider the set $\gamma' \setminus \gamma$ of points that are in $\gamma'$ but not in $\gamma$. It consists of finitely many connected components, each being a piecewise linear open arc.
Take any component $\gamma_j'$. There is an arc $\gamma_j$ on $\gamma$ so that the union of the closures of $\gamma_j$ and $\gamma_j'$ (each with suitable orientation)
bounds a disc in $A$. The same argument as before shows that   $\ell_{\phi,h}(\gamma_j)\leq \ell_{\phi,h}(\gamma_j')$.
Instead of $\gamma'$ we consider the closed piecewise smooth curve $\gamma''$ obtained by replacing each $\gamma_j'$ by $\gamma_j$ oriented so that $\gamma_j'$ and $\gamma_j$ are homotopic with fixed endpoints.
The horizontal $\phi$-length $\ell_{\phi,h}(\gamma'')$ of the new curve $\gamma''$ does not exceed $\ell_{\phi,h}(\gamma')$. All points of $\gamma''$ lie on $\gamma$. Each point of $\gamma$ is covered by at least one point of $\gamma''$, but some parts of $\gamma$ may be covered multiply. Hence $\ell_{\phi,h}(\gamma)\leq \ell_{\phi,h}(\gamma'')\leq \ell_{\phi,h}(\gamma')$. \hfill $\Box$

\bigskip

\noindent {\bf Proof of Theorem \ref{thm3.1}.} Let $\varphi_0$ be a non-periodic absolutely extremal self-homeomorphism of
a closed Riemann surface $X$ of genus $g\geq 2$, and let $\varphi$ be a self-homeomorphism of $X$ that is isotopic to $\varphi_0$. We will prove that the entropy of $\varphi$  is not smaller than $\frac{1}{2}\log K(\varphi_0)$.
Theorem \ref{thm3.1} will then be a consequence of Theorem \ref{thm3.2}.

Denote by $\phi$ the quadratic differential of  $\varphi_0$. It is holomorphic on $X$.
Let $\gamma$ be a simple closed curve in $X$ with base point $x_0$ that is contained in the $\phi$-regular part of  $X$, is transversal to the vertical foliation of $\phi$, and has positive $\ell_{\phi,h}$-length. By Theorem \ref{thm3.04} such a curve exists. By Theorem \ref{thm3.102} $\gamma$ is not contractible. Denote by $\alpha$ the class of $\gamma$ in $\pi_1(X,x_0)$.
Let $d_{hyp}$ be the hyperbolic metric on the universal covering $\tilde{X}\cong \mathbb{C}_+$. We use the same notation for the induced metric on $X$.
By Lemma \ref{lem3.108} and  Proposition \ref{prop3.102'}
the inequalities
\begin{equation}\label{eq3.125}
\ell_{hyp}(\widehat{\alpha}) \geq   c_1 \ell_{\phi}(\widehat{\alpha}) \geq c_1 \ell_{{\phi},h}(\widehat\alpha)= c_1\ell_{{\phi},h}(\gamma)>0
\end{equation}
hold for the free homotopy class $\widehat{\alpha}$ represented by ${\alpha}$.
\index{$\ell_{hyp}$} \index{$d_{hyp}$}

Denote by $\tilde{x}_0\in \tilde X$ the point corresponding to the identity ${\rm Id}_{x_0,x_0}$ in $\pi_1(X,x_0,x_0)$. Let $\tilde{\alpha}$ be the lift of $\alpha$ to the universal covering with initial point $\tilde{x}_0$.
Let as before $\alpha^{cov}$ be the covering transformation that maps $\tilde{x}_0$ to $\tilde{\alpha}(\tilde{x}_0)$. 
Since $d_{hyp}(\tilde{x}_0,\alpha^{cov}(\tilde{x}_0))$ is the infimum of the hyperbolic lengths $\ell_{hyp}(\tilde{\gamma}')$ over all curves $\tilde{\gamma}'$ in $\tilde X$ with initial point $\tilde{x}_0$ and terminal point $\alpha^{cov}(\tilde{x}_0)$ and $\ell_{hyp}(\tilde{\gamma}')= \ell_{hyp}({\gamma}')$,
we obtain by \eqref{eq3.125}
\begin{equation}\label{eq3.126}
d_{hyp}(\tilde{x}_0,\alpha^{cov}(\tilde{x}_0)) \geq c_1\ell_{{\phi},h}(\gamma)\,.
\end{equation}

Since $\varphi_0$ is a non-periodic absolutely extremal self-homeomorphism of $X$ with quadratic differential $\phi$, for any natural number $n$ the curve $(\varphi_0)^n(\gamma)$ is a simple closed curve in $X$ that avoids the $\phi$-singular points and is transversal to the vertical foliation. 
The mapping $(\varphi_0)^n$ expands the horizontal $\phi$-length of arcs by the factor $K^{\frac{n}{2}}$. This implies the following estimate of the horizontal $\phi$-length of  $(\varphi_0)^n(\gamma)$
\begin{equation}\label{eq3.127}
\ell_{\phi,h}((\varphi_0)^n(\gamma)) \geq  K^{\frac{n}{2}} \ell_{{\phi},h}(\gamma) \,.
\end{equation}
The loop
$(\varphi_0)^n(\gamma)$ represents the free homotopy class of $((\varphi_0)_{\#})^n(\alpha)$ (considered as element of the fundamental group $\pi_1(X,x_0)$).
Indeed, $(\varphi_0)^n(\gamma)$ lifts to the curve $(\tilde{\varphi_0})^n(\tilde{\gamma})$
in $\tilde X$ with initial point $(\tilde{\varphi_0})^n(\tilde{x}_0)$ and terminal point  $\big((\tilde{\varphi_0})^n(\alpha^{cov}(\tilde{x}_0))=\Big(\big ((\varphi_0)_{\#}\big)^n(\alpha)\Big)^{cov}\big((\tilde{\varphi_0})^n(\tilde{x}_0)\big) $. The last equality follows from equality \eqref{eq3.07f}.

Lemma \ref{lem3.108} and  Proposition \ref{prop3.102'} apply to the loop $(\varphi_0)^n(\gamma)$ in the same way as to the loop $\gamma$. Hence,
inequality \eqref{eq3.125} with  $(\varphi_0)^n(\gamma)$ instead of $\gamma$ holds, and inequality \eqref{eq3.126} is true with $\alpha^{cov}$ replaced by $(((\varphi_0)_{\#})^n(\alpha))^{cov}$.
Using inequality \eqref{eq3.127}, we obtain
\begin{equation}\label{eq3.129}
d_{hyp}(\tilde{x}_0,(((\varphi_0)_{\#})^n(\alpha))^{cov}(\tilde{x}_0)) \geq c_1\ell_{{\phi},h}((\varphi_0)^n(\gamma)) \geq c_1\, K^{\frac{n}{2}}\,\ell_{{\phi},h}(\gamma) \,.
\end{equation}

Consider the self-homeomorphism $\varphi$ of $X$ that is isotopic to $\varphi_0$.
Since by Theorem \ref{thm3.101}, Lemma \ref{lem3.5} and equality \eqref{eq3.112}
$$
h(\varphi) \geq \Gamma_{\varphi_{\#}}= \Gamma_{(\varphi_0)_{\#}}\geq
\limsup \frac{1}{n} \log d_{hyp}(\tilde{x}_0,( ((\tilde{\varphi_0})_{\#})^n(\alpha))^{cov}(\tilde{x}_0))\,,
$$
we obtain
$h(\varphi) \geq \frac{1}{2} \log K\,$.
The Theorem is proved.
\hfill $\Box$

\section{The entropy of mapping classes on second kind Riemann surfaces}
\label{sec:entropy.4}
The following Theorem \ref{thm3.5} allows us
to treat the entropy of
mapping classes of Riemann surfaces of second kind.
\begin{thm}\label{thm3.5}
Let $X$ be a connected closed Riemann surface with a set $E$ of
distinguished points. Assume that $X \backslash E$ is hyperbolic
(i.e. covered by ${\mathbb C}_+$). Let $\varphi$ be an absolutely extremal non-periodic self-homeomorphism of $X$ that fixes $E$ setwise. Suppose $z_0 \in
E$ is a fixed point of $\varphi$, $\varphi (z_0) = z_0$.

Then $\varphi$ is isotopic through self-homeomorphisms of $X$ that
fix $E$ to a homeomorphism $\varphi_0$ that fixes a topological
disc $\Delta'$ around $z_0$ pointwise and has the same entropy  $h(\varphi_0)
= h(\varphi)$ as $\varphi $. Moreover, the homeomorphism $\varphi_0$ can be chosen to represent any a priory given preimage under the mapping $\mathcal{H}_{z_0}:  \mathfrak{M}(X\setminus \Delta';\partial \Delta', E\setminus\{z_0\})\to
\mathfrak{M}(X; \{z_0\}, E\setminus\{z_0\})$ of the class of $\varphi$ in $\mathfrak{M}(X; \{z_0\}, E\setminus\{z_0\})$.
\end{thm}
(For the notation see Section \ref{sec:2.2}.)

\medskip

Notice that the complement $X_0$ in $X$ of an open disc around $z_0$
is a bordered Riemann surface \index{Riemann surface ! bordered}(the
closure of a Riemann surface of second kind), and the restriction
$\varphi _0 \mid X_0$ is a self-homeomorphism of $X_0$ with entropy
$h(\varphi _0 \mid X_0) = h(\varphi)$.

We will now prepare the proof of Theorem \ref{thm3.5}.

In a neighbourhood of the distinguished point $z_0$ of $X$ we consider distinguished coordinates $\zeta$ with $\zeta(z_0)=0$.  In these coordinates the quadratic differential of $\varphi$ has the form $\phi(\zeta)(d\zeta)^2=(\frac{a+2}{2})^2 \zeta^{a} (d\zeta)^2$. Take a disc $\Delta =\{|\zeta|< c\}$ in these coordinates. The following lemma says that the restriction $\varphi \mid X \setminus \{z_0\}$ has an extension to a self-homeomorphism of the ''real blowup'' of $X$ at
$z_0$.
\begin{lemm}\label{lemm3.6} In distinguished coordinates $\zeta=re^{i\theta}$ the
quotient $\frac{\varphi(re^{i\theta})}{|\varphi(re^{i\theta})|}$ does not depend on $r\in (0,c)$ and defines a homeomorphism
\begin{equation}\label{eq3.9'}
\mathfrak{h}(e^{i\theta})\stackrel{def}= \frac{\varphi(re^{i\theta})}{|\varphi(re^{i\theta})|}  \,
\end{equation}
of the unit circle onto itself.
\end{lemm}
\index{$\mathfrak{h}$}
\noindent {\bf Proof.} 
The distinguished
coordinates are unique up to multiplication by an $(a+2)^{\rm nd}$
root of unity. In other words
\begin{equation}\label{eq3.9}
\left( \zeta \, e^{\frac{2\pi i}{a+2}\cdot l} \right)^a \,\left( d
 \zeta \, e^{\frac{2\pi i}{a+2}\cdot l} \right)^2 = \zeta^a
(d\zeta)^2
\end{equation}
for any integer $l$, and these are the only holomorphic changes of
coordinates after which the quadratic differential has again the
canonical form.


The initial quadratic differential of $\varphi$ coincides with its
terminal quadratic differential. Hence in the distinguished
coordinates $\zeta$ the mapping $\varphi$ has the form
\begin{equation}\label{eq3.10} \varphi (\zeta) = {\mathcal Z}_a (\zeta) \cdot
e^{\frac{2\pi i}{a+2} \cdot \ell} \quad \mbox{for some integer
$\ell$},
\end{equation}
where \index{$\mathcal{Z}_a (\zeta)$}
\begin{equation}\label{eq3.11} {\mathcal Z}_a (\zeta) =\left( \frac{\zeta^{a+2} + 2k
\, \vert \zeta \vert^{a+2} + k^2 \, \bar\zeta^{a+2}}{1-k^2}
\right)^{\frac1{a+2}} \, .
\end{equation}
Here $k = \frac{K-1}{K+1}$, and $K = K(\varphi)$ is the
quasiconformal dilatation of $\varphi$. We take the root which is
positive on the positive real axis. (See e.g. \cite{Be1},
formula~(2.3) and Theorem~6 of \cite{Be1} or Chapter \ref{chapter2}, equation  \eqref{eq2.17}).

The mapping $\mathcal{Z}_a$ fixes each
horizontal bisectrix  $\{r \exp(\frac{2\pi j i}{a+2}),\, 0<r<c\} $,
$j=0,\ldots,a+1, $ and each vertical bisectrix $\{r \exp(\frac{2\pi (j +\frac{1}{2}) i}{a+2}),\, 0<r<c\} $,
$j=0,\ldots,a+1 $.
Consider a sector  $\mathfrak{s}^h_j= \{r \exp({ \theta i}): \, \theta \in  (\frac{2\pi j i}{a+2}, \frac{2\pi (j+1) i}{a+2}), \;0<r<c \} \,, j=0,\ldots,a+1,$ of $\Delta$ between two consecutive horizontal bisectrices. The flat coordinates $g^h_j$ on the sector  $\mathfrak{s}^h_j$ map the sector to a half-disc in the upper or lower half-plane. Recall that $g^h_j$ is  a suitable branch of the mapping $\zeta\to \zeta'\stackrel{def}=\zeta^{\frac{a+2}{2}}$, $g^h_j(re^{i\theta})= r^{\frac{a+2}{2}} e ^{i\frac{a+2}{2}\theta}$ for $re^{i\theta}\in \mathfrak{s}^h_j$. The image is in the upper half-plane if $j$ is even and in the lower half-plane if $j$ is odd.

Put $\zeta'=\xi'+ i \eta'= r'e^{i\theta'}$.
In coordinates $\zeta'$ the mapping $\mathcal{Z}_a$ has the form
\begin{equation}\label{eq3.12}
\zeta' =\xi' + i \eta'  \to K^{\frac12} \, \xi' + i \, K^{-\frac12} \, \eta' = \mathcal{Z}'_a(\zeta') \,
\end{equation}
with $K=\frac{1+k}{1-k}$.
Hence, for $\zeta' = r'e^{i\theta'},\, \theta' \in (0,\pi)\,
{\rm or}\; \theta' \in (\pi,2\pi), \;{\rm respectively},\;
0<r'< c^{\frac{2}{a+2}},$
the equality
\begin{align}\label{eq3.12'}
\frac{\mathcal{Z}'_a(r'e^{i\theta'})}{|\mathcal{Z}'_a(r'e^{i\theta'})|}=&
\frac{K^{\frac{1}{2}}
r' \cos \theta' + i r K^{-\frac{1}{2}}
r' \sin \theta'}{\sqrt{Kr'^2\cos^2(\theta')+K^{-1}r'^2\sin^2(\theta')}    }\nonumber\\
=&
\frac{K^{\frac{1}{2}}
\cos \theta' + i K^{-\frac{1}{2}}
\sin \theta'}{\sqrt{K \cos^2(\theta')+K^{-1}\sin^2(\theta')}    }
\end{align}
holds. The quotient $\frac{\mathcal{Z}'_a(r'e^{i\theta'})}{|\mathcal{Z}'_a(r'e^{i\theta'})|}$ does not depend on $r'$ and is denoted by $\mathfrak{h}'(e^{i\theta'})$. The mapping $\mathfrak{h}'$ defines a homeomorphism of $(0,\pi)$ ($(\pi,2\pi)$, respectively) onto itself and extends to a homeomorphism of $[0,\pi]$  ($[\pi,2\pi]$, respectively) onto itself. In coordinates $\zeta$ we obtain a homeomorphism $\mathfrak{h}$ from $[\frac{2\pi ji}{a+2},\frac{2\pi (j+1)i}{a+2}]$ onto itself, that fixes each endpoint. Moreover, $\mathcal{Z}_a(re^{i\theta})=|\mathcal{Z}_a(re^{i\theta})| \mathfrak{h}(e^{i\theta})$. The lemma follows from the equation \eqref{eq3.10}.
\hfill $\Box$

\bigskip

\index{$X^{z_0}$}
We will now define a homeomorphism from the ''real  blowup'' of $X$ at $z_0$ onto a subset of $X$. We put $\Delta_0=\{\zeta\in\Delta:\, |\zeta|<\frac{c}{2}\}$, and $X^{z_0} = X \setminus \overline{\Delta_0}$. \index{$X^{z_0}$} Then $X^{z_0} \subset X$ is a
Riemann surface of second kind which is homeomorphic to $X
\backslash \{ z_0 \}$. The closure $\overline{X^{z_0}}$ (the closure is taken in $X$) is a bordered Riemann surface and can be identified with the ''real blowup'' of $X$ at $z_0$.
We choose a diffeomorphism $\psi^{z_0}:X^{z_0} \to X
\backslash \{ z_0 \}$ which is the identity on $X \backslash \Delta$
and is defined on $X^{z_0} \cap \Delta$ in coordinates $\zeta$ by
the mapping 
\index{$\psi^{z_0}$}
\begin{equation}\label{eq3.8}
\left\{ \frac12 \, c < \vert \zeta \vert < c \right\} \ni \zeta \to
\psi^{z_0} (\zeta) = \alpha \, (\vert \zeta \vert) \cdot
\frac\zeta{\vert\zeta\vert} \in \Delta \backslash \{ 0 \}
\end{equation}
for a smooth strictly increasing function $\alpha : \left( \frac12
\, c,c \right) \to (0,c)$ with $\lim_{t \to \frac c2} \alpha (t) =
0$ and $\alpha (t) = t$ for $t$ close to $c$.
Then the self-homeomorphism
\begin{align*}
(\psi^{z_0})^{-1} \circ \varphi \circ \psi^{z_0}: X^{z_0}\toitself
\end{align*}
of $X^{z_0}$ extends
to a self-homeomorphism of $\overline{X^{z_0}}$ denoted by $\varphi^{z_0}$.

\begin{lemm}\label{lemm3.7} The following equality for entropies hold:
$h (\varphi^{z_0}) = h(\varphi)$.
\end{lemm}

\bigskip

\noindent {\bf Proof of Lemma \ref{lemm3.7}.}\\
\noindent {\bf The upper bound.} The inequality $h(\varphi) \leq h(\varphi ^{z_0})$ follows from Theorem~5 of \cite{AKM}. Indeed,
the mapping $\psi^{z_0}:X^{z_0} \to X \setminus \{z_0\}$ extends to a continuous mapping  $\widetilde{\psi^{z_0}}:\overline{X^{z_0}}\to X$
that maps the circle $\partial \Delta$ to $z_0$. We
introduce the equivalence relation $\sim$ on $\overline{X^{z_0}}$, for which $z\sim z'$ for different points $z,z'\in \overline{X^{z_0}}$ if and only if $z$ and $z'$ are points of $\partial \Delta_0$.  We denote by $\pi$ the canonical projection $\pi:  \overline{X^{z_0}}\to \overline{X^{z_0}}\diagup \sim   $.
The quotient  $\overline{X^{z_0}}\diagup \sim$ is homeomorphic to
$X$. A homeomorphism is given by the mapping $\mathring{\psi}^{z_0}$, $\mathring{\psi}^{z_0}(x\diagup \sim)=\psi^{z_0}(x)$ for $x\in X^{z_0}$ and $\mathring{\psi}^{z_0}(\partial X^{z_0}\diagup \sim)=z_0$. Then $\widetilde{\psi^{z_0}}=\mathring{\psi}^{z_0}\circ \pi$. The equality $\varphi \circ \widetilde{\psi^{z_0}}=\widetilde{\psi^{z_0}}\circ \varphi^{z_0}$ holds on $X^{z_0}$ by the definition of $\psi^{z_0}$. It is clear that it holds on the whole  $\overline{X^{z_0}}$.
We obtain the equality
\begin{align*}
\big((\mathring{\psi}^{z_0})^{-1}\circ\varphi \circ\mathring{\psi}^{z_0}\big)\circ \pi = \pi\circ \varphi^{z_0} \,.
\end{align*}
By Theorem 5 of \cite{AKM} the inequality
 $h(\varphi)= h\big((\mathring{\psi}^{z_0})^{-1}\circ\varphi \circ\mathring{\psi}^{z_0}\big)\leq h(\varphi ^{z_0})$ holds. (See also the beginning of the proof of Lemma \ref{lem3.4} of Section \ref{sec:entropy.1}.)

\medskip

\noindent {\bf The lower bound.}
We need to prove the opposite inequality $h(\varphi^{z_0}) \leq h(\varphi)$. The difficulty is that there is no refining sequence of covers of  $\overline{X^{z_0}}$ of the form $(\widetilde\psi^{z_0})^{-1}(\mathcal{A}_j))$ for covers $\mathcal{A}_j$ of $X$. Each cover $(\widetilde\psi^{z_0})^{-1}(\mathcal{A})$ with $\mathcal{A}$ a cover of $X$ contains an element that covers the whole circle $\partial X^{z_0}$.


\noindent{\bf Refining sequences of covers of $\overline{X^{z_0}}$.} We will find refining sequences of covers for  $\overline{X^{z_0}}$ as follows.
Let $N$ be a large natural number. For a positive number $\varepsilon < \varepsilon_0$ we take the open cover $\mathcal{A}_{\varepsilon}$ of $X$ consisting of $\varepsilon$-squares and $\varepsilon$-stars. We associate to $\mathcal{A}_{\varepsilon}$ a cover $\mathcal{A}_{\varepsilon,N}'$ of $X\setminus\{z_0\}$ which is described as follows. 
Consider $N$ radii in the disc $\Delta \subset X$ around $z_0$ which divide the disc into $N$ open sectors $S_j$ of equal angle.
We may choose the number $N$ and the radii so that
each sector $S_j$ is contained in a sector between a horizontal bisectrix and one of the closest vertical bisectrices. (Thus they divide each such sector 
into sectors of equal angle.)
For each $j$ we let $S'_j$ be an open sector which has an angle not exceeding twice the angle of $S_j$ and contains $\overline{ S_j} \setminus \{0\}$.
Then the $S'_j$ cover the punctured disc $\Delta \setminus \{z_0\}$. We may assume that the maximum over  $j=0, \ldots N,$ of the (Euclidean) length of the arcs
$\{e^{i\theta}: \frac{c}{2}e^{i\theta} \in S_j' \}$ of the unit circle goes to $0$ for $N\to \infty$.

Let $\varepsilon$ be small and $N$ big. The cover $\mathcal{A}'_{\varepsilon,N}$ of $X \setminus \{z_0\}$
consists of all sets in $\mathcal{A}_{\varepsilon}$ that do not intersect the disc $\Delta_0\stackrel{def}=\{|\zeta|< \frac{c}{2}\}$ around $z_0$, 
and all sets of the form $A \cap{S}'_j$ for any sector ${S}'_j$ and any set $A \in \mathcal{A}_{\varepsilon} $ that intersects  $\Delta_0$.

\begin{figure}[H]
\begin{center}
\includegraphics[width=65mm]{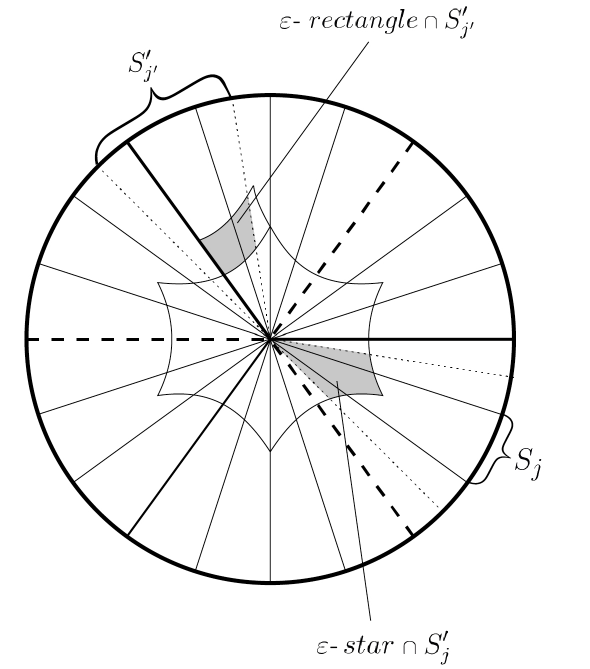}
\end{center}
\caption{The cover $\mathcal{A}'_{\varepsilon,N}$}\label{figen2}
\end{figure}

We define an open cover $\mathcal{A}_{\varepsilon,N}(z_0)$ of $\overline{X^{z_0}}$ as follows.
For each element $A \in\mathcal{A}'_{\varepsilon,N}$ that is not equal to the intersection of the $\varepsilon$-star at $z_0$ with a sector $S_j'$  we let the set $(\psi^{z_0})^{-1} (A)$ be an element of the cover $\mathcal{A}_{\varepsilon,N}(z_0)$. For each element $A \in\mathcal{A}'_{\varepsilon,N}$ that is an intersections of  the $\varepsilon$-star at $z_0$ with a sector $S_j'$
the boundary of the set $(\psi^{z_0})^{-1} (A)$ intersects the circle $\{|z-z_0|=\frac{c}{2}\}$ along an arc $\gamma_A$. We let the union $(\psi^{z_0})^{-1} (A) \cup \gamma_A$  be an element of the cover $\mathcal{A}_{\varepsilon,N}(z_0)$. These are all elements of the cover.
\index{$\mathcal{A}_{\varepsilon,N}(z_0)$ }

Take a sequence of $\varepsilon_k$'s that decrease to $0$, a sequence of $N_k$'s with $N_k$ dividing $N_{k+1}$ and  $N_k \to \infty$. Choose radii and sectors $S_{j_k}^{N_k}$ successively so that the $S_{j_1}^{N_1}$ subdivide the $\mathfrak{s}_j$ and for all $k>1$ the $S_{j_{k+1}}^{N_{k+1}}$ subdivide the $S_{j_k}^{N_{k}}$. Define sectors $({S_{j_{k}}^{N_{k}}})' \supset \overline{S_{j_k}^{{N_k}}} \setminus\{z_0\}$ as above and so that each
$({S_{j_{N+1}^{N_{k+1}}}})'$ is contained in some $({S_{j_k}^{N_{k}}})'$.
Consider for each $k$ the cover $\mathcal{A}_{\varepsilon_k,N_k}(z_0)$ of $\overline{X^{z_0}}$ obtained for the choice of $\varepsilon_k$, $N_k$, and $({S_j^{N_{k}}})'$.
Then the  $\mathcal{A}_{\varepsilon_k,N_k}(z_0)$
provide a refining sequence of open covers of $\overline{X^{z_0}}$.
This is not difficult to see, taking into account the choice of the $({S_j^{N_{k}}})'$, the fact that the sequence of coverings  $\mathcal{A}_{\varepsilon_k}$ of $X$ is refining, and the following observation.
Fix $k$ and consider the elements of the covers  $\mathcal{A}_{\varepsilon_k,N_k}(z_0)$ contained in $\overline{X^{z_0}} \cap \Delta$.
Let $\sigma_k$ be the supremum of the diameters (in coordinates $\zeta$ on $\Delta$) of such sets. Then the $\sigma_k$ tend to zero for $k\to \infty$, since in polar coordinates $\zeta= \varrho e^{i\theta}$ on $\Delta\cap \overline{X^{x_0}}$
each set $(\psi^{z_0})^{-1}(A)$ with $A \in
\mathcal{A}_{\varepsilon,N}'$ has width comparable to $\varepsilon_N$ in the $\varrho$-direction and width equal to the maximal angle of the ${S_j^N}'$ in the $\theta$ direction.

\noindent{\bf An upper bound for the entropy $h(\varphi^{z_0},\mathcal{A}_{\varepsilon,N}(z_0)) $.}
We will fix any sufficiently small number $\varepsilon$ and any large integer $N$ and give an upper bound of the entropy $h(\varphi^{z_0},\mathcal{A}_{\varepsilon,N}(z_0)) $ of $\varphi^{z_0}$ with respect to the cover $\mathcal{A}_{\varepsilon,N}(z_0)$ of $\overline{X^{z_0}}$.

For each set $A$ in the above defined cover $\mathcal{A}_{\varepsilon}$ of $X$ and one of the $N$ sectors $S_j^N$ of $\Delta\setminus \{z_0\}$
we define $A \cap^* S_j^N$ to be equal to $A$ if $A$ does not intersect $\Delta_0$ and equal to $A \cap S_j^N$ if $A$ intersects $\Delta_0$.
We prove first the following claim. \\
{\it For each $(\lambda^{-n}\varepsilon,\varepsilon)$-rectangle and each
$(\lambda^{-n}\varepsilon,\varepsilon)$-star written in the form
\begin{equation}\label{eq3.8a}
A_0\cap \ldots \cap\varphi^{-n}(A_n) \in \mathcal{A}_{\varepsilon}\vee \ldots \vee  \varphi^{-n}(\mathcal{A}_{\varepsilon})
\end{equation}
with either all $A_j$ being $\varepsilon$-squares or all $A_j$ being $\varepsilon$-stars,
there are at most $1+\frac{(n+1)N((n+1)N+1)}{2}$ different non-empty sets of the form
\begin{align}\label{eq3.8b}
A\stackrel{def}=\big(A_0\cap^*  S_{j_0}^N\big) \bigcap \ldots  \bigcap\varphi^{-n}\big(A_n\cap^* S_{j_n}^N\big)
\end{align}
where each $S_{j_l}^N$ is one of the $N$ sectors of $\Delta$.}

The claim is proved as follows. Suppose first that the set \eqref{eq3.8a}
is a $(\lambda^{-n}\varepsilon,\varepsilon)$-rectangle, i.e.
it is
the image of a $\phi$-rectangle $\mathcal{F}:R\to A_0\cap \ldots \cap\varphi^{-n}(A_n)$  for a rectangle $R$ in the complex plane with sides parallel to the axes. 
If a set $A_l$ intersects $\Delta_0$
then by a similar argument as used in the proof of Lemma \ref{lem3.109a} it is contained in a sector $\mathfrak{s}_{j_l}^h$ (or  $\mathfrak{s}_{j_l}^v$) of $\Delta$ between two consecutive horizontal bisectrices (or between two consecutive vertical bisectrices). A branch of the mapping $\zeta\to \zeta^{\frac{2}{a+2}}$ defines flat coordinates on the sector and maps each of the $N$ radii that are contained in the sector 
to a relatively closed straight line segment in the image. The mapping $\varphi$ is real linear in flat coordinates. Hence the mapping $\varphi^{-l}$ takes each of the $N$ radii that are contained in the sector to a relatively closed curve in the image which is a straight line segment in flat coordinates.
This implies that each set of the form \eqref{eq3.8b} is obtained as follows.

Take the rectangle $R$ in the plane which is the preimage $\mathcal{F}^{-1}(A)$ of a set $A$ of the form 
\eqref{eq3.8a}. For each $l$ the collection of preimages $\mathcal{F}^{-1}(A\bigcap^*\varphi^{-l}(A_l\cap  S_{j_l}^N))$ is the collection of intersections of $R$ with connected components of the complement of the union of no more than $N$ lines in the complex plane. We have to count non-empty sets each of which is the intersection over $l=0,\ldots ,n,$ of an element of the collection of preimages $\mathcal{F}^{-1}(A\bigcap\varphi^{-l}(A_l\cap^*  S_{j_l}^N))$.

Each such intersection $Q$ is a connected component of the intersection of $R$ with a connected component $Q'$ of the complement of the union of no more than $(n+1)N$ lines in the complex plane. Indeed, $Q$ is contained in such a component $Q'$, and since the boundary of $Q$ is the union of segments of some of the lines, it has no boundary points in $Q'$, hence it coincides with $Q'$.

The estimate of the number of connected components of the complement of the union of $N$ different real lines in the complex plane can be given by induction on the number $N$ of lines. If $N$ equals $1$ the number is obviously equal to $2$. Suppose an upper bound for the number is found for $N$ lines. The observation for $N+1$
lines is the following. The $(N+1)$-st line intersects a number of lines among the $N$ previous lines, and each of them is intersected once. The intersection points with the $N$  previous lines divide the $(N+1)$-st
\newline
 line into no more than $N+1$ connected components. Each connected component of the last line is contained in exactly one connected component of the complement of the first  $N$ lines in $\mathbb{C}$.
This component is divided into two parts by the last line. Hence, adding an $(N+1)$-st line increases the number of connected components of the lines at most by $N+1$. By induction, the number of connected components of the complement of $N$ lines in $\mathbb{C}$ does not exceed $1+ \frac{N(N+1)}{2}$. We proved the claim for the case when $A$ is a
$(\lambda^{-n}\varepsilon,\varepsilon)$-rectangle.

If the set \eqref{eq3.8a} is a $(\lambda^{-n}\varepsilon,\varepsilon)$-star and intersects $\Delta_0$,
then all $A_j$ are stars at $z_0$, and for all $A_j$ that are stars at $z_0$ we obtain at most $N$ radii in each $A_j$.
Hence, for the stars in $\mathcal{A}_{\varepsilon} \vee \varphi^{-1}(\mathcal{A}_{\varepsilon}) \vee \ldots \vee \varphi^{-n}(\mathcal{A}_{\varepsilon})$
there are at most $(n+1)N$ different non-empty sets of the form \eqref{eq3.8b}.
The claim is proved.

For each small $\varepsilon >0$ there is a finite number $C_1$ depending on $\varepsilon$, such that
the set $X\setminus\{z_0\}$ can be covered by $C_1$ sets, each either an $\varepsilon$-square
or an
$\varepsilon$-star. We may cover each  $\varepsilon$-square by $[\lambda^n+1]$ sets that are  $(\lambda^{-n}\varepsilon,\varepsilon)$-rectangles. Consider an $\varepsilon$-star centered at a singular point of order $a'$. Each of the $2a'+2$ sectors of the  $\varepsilon$-star between a horizontal and the nearest vertical bisectrices is an open $\varepsilon$-square. Hence, the punctured star can be covered by $(2a'+2)\cdot [\lambda^n+1]$  $(\lambda^{-n}\varepsilon,\varepsilon)$-rectangles, and the puncture is covered by a $(\lambda^{-n}\varepsilon,\varepsilon)$-star.

By the proof of Proposition \ref{prop3.110} 
each $(\lambda^{-n}\varepsilon, \varepsilon)$-rectangle can be written in the form \eqref{eq3.8a} for $\varepsilon$-squares $A_l$, and each  $(\lambda^{-n}\varepsilon,\varepsilon)$-star can be written in the form \eqref{eq3.8a} for $\varepsilon$-stars $A_l$.
Hence, there is a finite constant $C_2$, such that for each natural number $n$ the set $X$ can be covered by no more than $C_2 \cdot (\lambda^n+1)$ sets of the form \eqref{eq3.8a}.  By the claim, $X\setminus\{z_0\}$ can be covered by no more than $C_2\big(1+\frac{(n+1)(N(n+1)+1)}{2}\big) \cdot (\lambda^n+1) $ different non-empty sets of the form
\eqref{eq3.8b} with each $S_l^N$ replaced by the slightly bigger $({S_l^N})'$. These sets form a subcover of $\mathcal{A}_{\varepsilon,N}' \vee \varphi^{-1}(\mathcal{A}_{\varepsilon,N}') \ldots \vee \vee  \varphi^{-n}(\mathcal{A}_{\varepsilon,N}')$.

We obtain a cover of $\overline{X^{z_0}}$ by considering for each element of $\mathcal{A}_{\varepsilon,N}'$ the respective element of $\mathcal{A}_{\varepsilon,N}(z_0)$. This gives the inequality
\begin{align*}
& \mathcal{N}(\mathcal{A}_{\varepsilon,N}(z_0) \vee \varphi^{-1}(\mathcal{A}_{\varepsilon,N}(z_0)) \vee \ldots \vee \varphi^{-n}(\mathcal{A}_{\varepsilon,N}(z_0)))\\
\leq &C_2\Big(1 +\frac{(n+1)N((n+1)N+1)}{2}\Big)\cdot
(\lambda^n+1)\,.
\end{align*}
Hence,
\begin{align}\label{eq3.141}
& h(\varphi^{z_0})\nonumber \\
= &  \limsup \frac{1}{n}  \log\Big(\mathcal{N} (\mathcal{A}_{\varepsilon,N}(z_0) \vee \varphi^{-1}(\mathcal{A}_{\varepsilon,N}(z_0)) \vee \ldots \vee \varphi^{-n}(\mathcal{A}_{\varepsilon,N}(z_0))\Big)\nonumber \\
\leq
& \limsup \frac{1}{n}  \log\Big(C_2\big(1+\frac{N(n+1)(N(n+1)+1)}{2}\big)\cdot(\lambda^n+1)\Big) = \log \lambda\,.
 \end{align}
\noindent Lemma \ref{lemm3.7} is proved. \hfill $\Box$

\bigskip

The restriction of the homeomorphism $\varphi^{z_0}$ to $\partial X^{z_0}$  is not equal to the identity mapping on  $\partial X^{z_0}$. The Lemmas \ref{lemm3.10} and \ref{lemm3.11} below produce a self-homeomorphism of a slightly larger bordered Riemann surface whose restriction to the boundary is the identity mapping.

\begin{lemm}\label{lemm3.10} For each integer $\ell$ there exists a self-homeomorphism
$\;\tilde\varphi\;$ of $\left\{ \frac r4 \leq \vert \zeta \vert \leq
\frac r2 \right\} \subset \Delta \subset X$ of entropy zero which
equals $\varphi^{z_0}$ on $\left\{ \vert \zeta \vert = \frac r2
\right\}$ and is equal to multiplication by $e^{\frac{2\pi
i\ell}{2(a+2)}}$ on $\left\{ \vert \zeta \vert = \frac r4 \right\}$ for the integer $\ell$.
\end{lemm}

Recall that $r\exp(\frac{2\pi j i}{2(a+2)}),\, j=1,\ldots, 2(a+2), 0<r<c$,  are the points on the vertical and horizontal bisectrices. Consider the arc $\gamma_j = \Bigl\{ \frac r2 \, e^{i\theta} :
\frac{2\pi j}{2(a+2)} \leq \theta \leq \frac{2\pi(j+1)}{2(a+2)}
\Bigl\}$, $j=0,1,\ldots , 2(a+2)-1$.
We will obtain Lemma
\ref{lemm3.10} from the following lemma.
\begin{lemm}\label{lemm3.12} Let $\psi_j$ be a
self-homeomorphism of $\gamma_j$ which fixes the endpoints of
$\gamma_j$ pointwise and does not fix any other point. Then there is
a self-homeomorphism $\tilde{\psi_j}$ of the truncated sector
$\bar\Omega_j = \left\{ \rho \, e^{i\theta} : \frac r4 \leq \rho
\leq \frac r2 , \frac r2 \, e^{i\theta} \in \gamma_j \right\}$ which
equals the identity on the boundary part $\partial \, \Omega_j
\backslash {\rm Int} \, \gamma_j$ of $\partial \, \Omega_j$, equals
$\psi_j$ on $\gamma_j$, and has entropy zero.
\end{lemm}


\noindent {\bf Proof.}
We may assume that $j=0$. Put $\varrho(t)=\frac{1}{2} t + \frac{1}{4}(1-t)$, $t\in[0,1]$. Then $\varrho(0)=\frac{1}{4}$,  $\varrho(1)=\frac{1}{2}$, and $\varrho$ is a linear homeomorphism from $[0,1]$ onto $[\frac{1}{4},\frac{1}{2}]$. Put $I\stackrel{def}=[0,\frac{2\pi}{2a+2}]$. The mapping
\begin{align*}
[0,1]\times I\ni (t,\theta) \xrightarrow {g} \varrho(t)e^{i\theta} \in \overline{\Omega_0}
\end{align*}
is a homeomorphism. The segment $\{1\}\times I$ is mapped homeomorphically onto
$\gamma_0$. Put $\Psi= g^{-1}\circ \psi_0\circ g\mid \{1\}\times I$.
The mapping $\Psi$ is a self-homeomorphism of $\{1\}\times I$ that is strictly increasing in the second variable.
We need to find a self-homeomorphism $\tilde{\Psi}$ of  $[0,1]\times I$ such that $\tilde{\Psi}\mid \{1\}\times I= \Psi$, the restriction of $\tilde{\Psi}$ to the rest of the boundary of $[0,1]\times I$ is the identity, and the entropy of $\tilde{\Psi}$ equals zero. Then the mapping $\tilde{\psi}_0\stackrel{def}=g\circ \tilde{\Psi}\circ g^{-1}$ satisfies the
requirement of the lemma for $j=0$.

We put
\begin{equation}\label{eq3.31}
\tilde{\Psi}(t ,\,\theta) = (t,\; t\,\,{\Psi}(\theta)\,  +\, (1-t)\,\theta ) , \quad (t,\theta)
\in [0,1] \times I\, ,
\end{equation}
and $\Psi_t(\theta)= t\,\,{\Psi}(\theta)\,  +\, (1-t)\,\theta$. 
The mapping $\tilde{\Psi}$ fixes the coordinate $t$, hence, it fixes each maximal vertical segment in $[0,1]\times I$, and maps the horizontal line segment 
with endpoints $(0,\theta)$ and  $(1,\theta)$ onto the line segment with endpoints $(0,\theta )$ and $(1,\Psi(\theta))$.

We will prove now that the homeomorphism $\tilde{\Psi}$ has
entropy zero. We choose open covers of $[0,1]\times I$ as follows. For a natural number $N$ we consider $N$ vertical line segments in $[0,1]\times I$ which divide $[0,1]\times I$ into $N+1$ vertical strips of
equal width, and $N$ horizontal line segments in $[0,1]\times I$  which divide  $[0,1]\times I$ into $N+1$ horizontal strips of
equal width. The complement of the lines in $[0,1]\times I$ consists of $(N+1)^2$
rectangles $R_l$. Let $\mathcal{A}_N$ be the cover of $[0,1]\times I$ by open $(N+1)^2$ rectangles each containing the closure of one of the $R_l$ and having diameter close to that of $R_l$.
For a suitable sequence of $N$'s we get a refining sequence of covers of $[0,1]\times I$. We prove now that for each $N$ the entropy $h(\tilde{\Psi},\mathcal{A}_N)$ is equal to zero.

The mapping ${\Psi}$ has no fixed point on the interior ${\rm Int} I$ of $I$. Hence, either ${\Psi}(\theta)> \theta$ for each $\theta\in  {\rm Int} I $, or ${\Psi}(\theta)< \theta$ for each $\theta$ in ${\rm Int} I$.
We assume that the first case holds. (In the second case we conjugate the mapping by $(t,\theta)\to (t,-\theta)$.)
Then for each $\theta \in  {\rm Int} I $ the mapping $t\to \Psi_t(\theta), \, t\in I,$ is strictly increasing.

Let $\Psi_t^k$ be the $k$-th iterate of $\Psi_t$ (considered as function of $\theta$ for fixed parameter $t$). Then $\tilde{\Psi}(t,\theta)=(t,\Psi_t(\theta))$, and
for the $k$-th iterate $(\tilde{\Psi})^k$ we have by induction
$$
\tilde{\Psi}^k\big(t,\theta\big)=\tilde{\Psi}\big(t,\Psi_t^{k-1}(\theta)\big)=
\big(t, \Psi_t^{k}(\theta)\big)\,.
$$

We claim that for each $k\geq 1$ the mapping $t\to \Psi_t^{k}(\theta))$
is strictly increasing for each $\theta $ in the interior of $I$. This is clear for the linear mapping $t\to \Psi_t(\theta)$ for $\theta \in {\rm Int} I$. By induction the claim for $k+1$ is obtained as follows.

Suppose we know that for all $0<t_1<t_2<1$ and $\theta\in {\rm Int} I$
the inequality
$\Psi_{t_2}^k(\theta)>\Psi_{t_1}^k(\theta)$ holds.
We have for all $t\in [0,1]$
\begin{align}\label{eq3.31'a}
\Psi_{t}^{k+1}(\theta)= & t\,{\Psi}\big(\Psi_{t}^k(\theta)\big)\,+\, (1-t)\,\Psi_{t}^k(\theta)\,.
\end{align}

Since ${\Psi}$ strictly increases and the mapping $t\to \Psi_{t}^{k}(\theta)$ strictly increases, we obtain
\begin{equation}\label{eq3.31'}
t\,{\Psi}\big(\Psi_{t_1}^k(\theta)\big))\,+\, (1-t)\,(\Psi_{t_1})^k(\theta)\,<\,
t\,{\Psi}\big(\Psi_{t_2}^k(\theta)\big)\,+\, (1-t)\,\Psi_{t_2}^k(\theta)\,
\end{equation}
for each $t\in[0,1]$, $\theta\in {\rm Int} I$, and $t_1,t_2\in [0,1]$, $t_1<t_2$.
Since ${\Psi}(\theta')>\theta'$ for each $\theta'$ in the interior of $I$,
we obtain 
the inequality
\begin{align}\label{eq3.31'b}
t_1\,{\Psi}\big(\Psi_{t_2}^k(\theta)\big)\,+\, (1-t_1)\,\Psi_{t_2}^k(\theta)\,<\,t_2\,{\Psi}\big(\Psi_{t_2}^k(\theta)\big)\,+\, (1-t_2)\,\Psi_{t_2}^k(\theta)\,.
\end{align}
Combining the last inequality \eqref{eq3.31'b} with inequality \eqref{eq3.31'} for $t=t_1$ and  equality \eqref{eq3.31'a} we obtain the claim.

The claim implies that for each natural number $k$ and for each $\theta$ in the interior of $I$ 
the curve
\begin{equation}\label{eq3.32}
\mathcal{C}_k^{\theta}\stackrel{def}=  \{\big(t, \Psi_t^k(\theta)\big)\,, t \in[0,1]\}= \tilde{\Psi}^k\big([0,1]\times \{\theta\}\big)
\end{equation}
intersects each vertical line and each horizontal line in $[0,1]\times I$
at most once. Since $\mathcal{C}_{k+l}^{\theta}= \tilde{\Psi}^k(\mathcal{C}_{l}^{\theta})   $
the intersection of two curves of form \eqref{eq3.32} consists of at most one point. The intersection behaviour of the union of the set of vertical and horizontal lines and the set curves $\mathcal{C}_{l}^{\theta}$ is the same as the intersection behaviour of straight lines.
A similar argument as in the proof of Lemma \ref{lemm3.7} proves that the entropy $h(\tilde{\Psi},\mathcal{A}_N)$ equals zero for each $N$. Hence, $h(\tilde{\Psi})=0$. Lemma \ref{lemm3.12} is proved. \hfill $\Box$

\bigskip

The following lemma states in particular that for a bordered Riemann surface $\overline X$ one can change a mapping class $\mathfrak{m}\in\mathfrak{M}(\overline {X}, \partial X)$ without increasing entropy  by a product of Dehn twists about circles that
are homologous to boundary curves.

\begin{lemm}\label{lemm3.11} Let $\hat\varphi$ be a twist on an annulus
$\left\{ \frac r8 \leq \vert \zeta \vert \leq \frac r4 \right\}$, more precisely, for some real constant $\nu$ we have $\hat\varphi
(\zeta) = \zeta \cdot e^{i\nu \cdot \left(\log \vert \zeta \vert  -
\log \frac r8 \right)}$. Then $h(\hat\varphi) = 0$.
\end{lemm}

\noindent {\bf Proof.}
Denote by $A$ the annulus $A= \left\{\, \frac{r}{8}\leq |\zeta| \leq \frac{r}{4}
\right\}\,$. Let $N$ be a natural number. Divide $A$ by
$N -1\,$ circles $\,{\rm Circ}_j= \left\{ \, |\zeta|= \frac{r}{8}
\cdot (1+\frac{j}{N})\, \right\}, \; j=1,\ldots,N-1,\;$ and
$N\,$ radii ${\rm Rad} _j\,= \left\{\,\rho \cdot e^{\frac{2\pi i
}{N}j}:\, \frac{r}{8} \leq \rho \leq \frac{r}{4}\right\}, \;
j=1,\ldots,N $ into a collection  $\mathcal{Q}^N$ of connected open sets $Q_l^N$. Locally the circles and the radii are straight line segments in logarithmic coordinates, and the $Q_l^N$ are open rectangles in these coordinates. Similarly as in the proof of Lemma \ref{lemm3.7} we take for each $l$ an open rectangle $({Q_l^N})'$ in local logarithmic coordinates that contains the closure $\overline{Q_l^N}$ and has diameter not exceeding twice the diameter of $Q_l^N$. We obtain an open cover $\mathcal{A}_N$ of the annulus $A$. For an increasing sequence $N_k$ of natural number with $N_{k+1}$ being an integer multiple of $N_k$ and a suitable choice of the $({Q_l^{N_k}})'$ we obtain a refining sequence of open covers.

For each $N$ the annulus can be covered by the collection of sets $(Q_0^N)'\cap\ldots \cap\hat{\varphi}^{-n}((Q_n^N)')$ for which the sets
$Q_0^N\cap\ldots \cap\hat{\varphi}^{-n}(Q_n^N)$ 
( $Q_l^N \in \mathcal{Q}^N$) are not empty. These sets $(Q_0)'^N\cap\ldots  \cap\hat{\varphi}^{-n}((Q_n^N)')$ are elements of a subcover
of the cover $\mathcal{A}_N\vee\ldots\vee\hat{\varphi}^{-n}(\mathcal{A}_N)$.    
We will estimate the number of these sets for each $N$ and $n$.

Each iterate $\hat{\varphi}^{-\ell},\, \ell=1,\ldots, n,$ is real linear in local logarithmic coordinates. It maps each circle ${\rm Circ} _j, \,
j=1,\ldots,N-1,\,$ onto itself and each radius ${\rm
Rad} _j,\, j=1,\ldots N,$ onto the curve
\begin{equation}\label{eq3.37}
\hat \varphi ^{-\ell}({\rm Rad} _j) = \left\{\, e^x \cdot e^{i
\frac{2\pi}{N} \, j + \, i \cdot \ell \cdot \nu \left( x - \log
\frac r8 \right)} : \log \frac r8 \leq x \leq \log \frac r4
\right\}.
\end{equation}

Each curve \eqref{eq3.37} intersects each given set $Q_l^N\in \mathcal{Q}^N$ along at most $c(\nu,N)\cdot \ell$ connected connected components, each a straight line segment in logarithmic coordinates on  $Q_l^N$. 
Hence, as in the proof of Lemma \ref{lemm3.7} the union of the $\hat \varphi ^{-\ell}({\rm Rad} _j), \, j=1,\ldots, N,\, \ell=1,\ldots ,n,\,$ intersects the given set $Q_l^N$ along at most $c(\nu,N) (N+\ldots +Nn) \leq c'(\nu,N) n^2$ straight line segments in logarithmic coordinates. They divide $Q_l^N$ into at most $C'(\nu,N) n^4$ connected components. Hence, the number
of non-empty sets of the form $Q_0^N\cap\ldots \hat{\varphi}^{-n}(Q_n^N)$
for sets $Q_l^N\in \mathcal{Q}^N$ does not exceed $C(\nu,N)n^4$. Here $c(\nu,N), c'(\nu,N), C'(\nu,N),$ and $C(\nu,N)$ are positive constants depending only on $\nu$ and $N$.

Hence, $\mathcal{N}(\mathcal{A}_N\vee\ldots\vee \hat{\varphi}^{-n}(\mathcal{A}_N))\leq C(\nu,N) n^4$. This implies that the entropy of
$\hat\varphi$ with respect to the cover $\mathcal{A}_N$ does not exceed
$\overline\lim_{n \to \infty} \, \frac1n \log (C(\nu , N) \cdot
n^4) = 0$. The estimate is obtained for each element of a refining
sequence of open covers of the annulus. Hence $h(\hat\varphi) = 0$.
\hfill $\Box$

\medskip

\noindent {\bf Proof of Lemma \ref{lemm3.10}.} Let $\ell$ be the integer from equality \eqref{eq3.10} and let $k$ be the
smallest positive integer for which $k\cdot \ell$ is a multiple of
$2(a+2)$. Then $\varphi^{z_0}$ permutes the $\gamma_j$ in cycles of
length $k$ and $(\varphi^{z_0})^k$ fixes each $\gamma_j$ setwise. In
particular, for each $j$ the mapping $(\varphi^{z_0})^k \mid \gamma_j$
satisfies the
conditions stated for the function $\psi_j$
of Lemma \ref{lemm3.12}.
Let $
\gamma_{\ell_1}
\overset{\varphi^{z_0}}{-\!\!\!-\!\!\!-\!\!\!\longrightarrow}
\gamma_{\ell_2} \longrightarrow \ldots \longrightarrow
\gamma_{\ell_k}
\overset{\varphi^{z_0}}{-\!\!\!-\!\!\!-\!\!\!\longrightarrow}
\gamma_{\ell_1}$ be one of the cycles for $\varphi^{z_0}$. For each
$\ell_j$, $j = 1,\ldots , k-1$, we take any homeomorphism
$\tilde\varphi_{\ell_j} : \bar\Omega_{\ell_j} \to
\bar\Omega_{\ell_{j+1}}$ which equals $\varphi^{z_0} \mid
\gamma_{\ell_j}$ on $\gamma_{\ell_j}$ and equals multiplication by
$e^{\frac{2\pi i \ell}{2(a+2)}}$ on the rest of the boundary of
$\bar\Omega_{\ell_j}$. Let $\psi_{\ell_1}$ be the self-homeomorphism
of $\bar\Omega_{\ell_1}$ obtained by applying Lemma \ref{lemm3.12}
to $(\varphi^{z_0})^k \mid \gamma_{\ell_1}$. We put
$\tilde\varphi_{\ell_k} = \psi_{\ell_1} \circ
(\tilde\varphi_{\ell_{k-1}} \circ \ldots \circ
\tilde\varphi_{\ell_1})^{-1}$. Then on $\gamma_{\ell_k}$ the equality
$\tilde\varphi_{\ell_k}=(\varphi^{z_0})^k\circ (\varphi^{z_0})^{-k+1}=\varphi^{z_0}$ holds. Denote by $\tilde{\varphi}_{cyc}$
the mapping that is defined by $\tilde\varphi_{\ell_{j}} $ on $\bar\Omega_{\ell_j}$ for $j=1,\ldots,k$.

The mapping that equals $\psi_{\ell_1}=\tilde\varphi_{\ell_k} \circ
\ldots \circ \tilde\varphi_{\ell_2} \circ \tilde\varphi_{\ell_1}$ on $\bar\Omega_{\ell_1}$ and $\tilde\varphi_{\ell_{j-1}} \circ \ldots \circ
\tilde\varphi_{\ell_1} \circ \tilde\varphi_{\ell_k} \circ \ldots
\circ \tilde\varphi_{\ell_{j+1}} \circ \tilde\varphi_{\ell_j}$ on $\bar\Omega_{\ell_j}$ has entropy zero, since by Lemma \ref{lemm3.11} $h(\psi_{\ell_1}) = 0$ and the mapping on $\bar\Omega_{\ell_j}$ is coniugate to  $\psi_{\ell_1}$. Moreover, the mapping
is equal to the $k$-th iterate  $\tilde{\varphi}_{cyc}$.
Hence, $h(\tilde{\varphi}_{cyc})=0$.

\smallskip

Proceed in the same way with all cycles of $\gamma_j$ under
iteration by $\varphi^{z_0}$.

\smallskip

The obtained mappings $\tilde\varphi_j$ map $\bar\Omega_j$
homeomorphically onto $e^{\frac{2\pi i \ell}{2(a+2)}} \,
\bar\Omega_j$. They match together to give a well-defined
self-homeomorphism $\tilde\varphi$ of the annulus $\left\{ \frac r4 \leq \vert
\zeta \vert \leq \frac r2 \right\}$ which equals $\varphi^{z_0}$ on
$\left\{ \vert \zeta \vert = \frac r2 \right\}$ and equals
rotation by $e^{\frac{2\pi i \ell}{2(a+2)}}$ on $\left\{ \vert \zeta
\vert = \frac r4 \right\}$. The $\psi_j$ match together to give a
self-homeomorphism $\psi$ of the annulus such that $\tilde\varphi^k
= \psi$. Since $h(\psi) = 0$ we have $h(\tilde\varphi)=0$. \hfill
$\Box$

\bigskip

\noindent{\bf Proof of Theorem \ref{thm3.5}.}
Put $\varphi_0 = \varphi^{z_0}$ on $\overline{X^{z_0}}$, take $r=c$, and put
$\varphi_0$ equal to the mapping $\tilde\varphi$ of Lemma
\ref{lemm3.10} on the annulus $\left\{ \frac c4 \leq \vert \zeta
\vert \leq \frac c2 \right\}$ with distinguished coordinates around
$z_0$. Let $\varphi_0$ be equal to the homeomorphism $\hat\varphi$
of Lemma \ref{lemm3.11} with $\nu = -\frac{2\pi \ell}{2(a+2)\cdot\log
2 }$ on the annulus $\left\{ \frac c8 \leq \vert \zeta \vert \leq
\frac c4 \right\}$ and equal to the identity on $\left\{ \vert \zeta
\vert \leq \frac c8 \right\}$. Then $\varphi_0$ is isotopic to $\varphi$ through self-homeomorphisms of $X$ that fix $E$,   $\varphi_0$ fixes the disc $\Delta'\stackrel{def}= \left\{\vert \zeta \vert \leq
\frac c8 \right\}$ around $z_0$ and has the same entropy as $\varphi$.
It follows that the
isotopy classes of the two mappings $\varphi$ and  $\varphi_0$ differ by a Dehn twist about a curve that is homologous to $\partial \Delta'$. By Lemma \ref{lemm3.11}  the homeomorphism $\varphi_0$ can be chosen to represent any a priory given preimage  of the class of $\varphi$ in $\mathfrak{M}(X; \{z_0\}, E\setminus\{z_0\})$ under the mapping $\mathcal{H}_{z_0}:  \mathfrak{M}(X\setminus \Delta';\partial \Delta', E\setminus\{z_0\})\to
\mathfrak{M}(X; \{z_0\}, E\setminus\{z_0\})$.
Theorem \ref{thm3.5} is proved. \hfill $\Box$

\bigskip

The following theorem concerns the slightly more general case when
the self - homeomorphism $\varphi$ of a Riemann surface of first
kind is changed without increasing entropy on the union of discs
around points of a subset of the set of distinguished points (rather
than on a single disc around a fixed distinguished point).

\medskip

\begin{thm}\label{thm3.6} Let $X$ be a connected closed
Riemann surface with a set $E$ of distinguished points. Assume that
$X \backslash E$ is hyperbolic. Let $\varphi$ be a
non-periodic absolutely extremal self-homeomorphism
of $X$ which fixes $E$ setwise. Suppose there is
a $\varphi$-invariant subset $E' \subset E$. 

\smallskip

Then $\,\varphi\,$ is isotopic through self-homeomorphisms which fix
$\,E\,$ setwise to a self-homeomorphism $\varphi_0$ of the same
entropy $h(\varphi_0) = h(\varphi)$ with the following property.

\smallskip

For each $z \in E'$ there is a closed round disc $\overline\delta_z$
in distinguished coordinates for the quadratic differential of
$\varphi$, such that $z$ is the center of $\delta_z$, and
$\varphi_0$ maps each $\delta_z$, $z \in E'$, conformally onto
another disc of the collection and maps the center of the source
disc to the center of the target disc. Moreover, if a
$\varphi$-cycle of points in $E'$ has length $k$, then the iterate
$\varphi_0^k$ fixes pointwise the disc $\delta_z$ around each point
$z$  of the cycle.
\end{thm}
Denote by $X'$ the bordered Riemann surface which is the complement of the union of the open discs $\cup_{z \in E'}\delta_z$.
The isotopy class of the restriction $\varphi_0|X'$ of the mapping $\varphi_0$ of Theorem \ref{thm3.6} is determined by $\varphi$ up to a product of powers of Dehn twists around simple closed curves that are homologous to
the boundary circles of $X'$. The mapping $\varphi_0$ of Theorem \ref{thm3.6} can be chosen so that the restriction $\varphi_0|X'$ represents the class corresponding to an a priory chosen product of powers of Dehn twists about the boundary circles.
The proof of Theorem \ref{thm3.6} follows along the same lines as the proof of
Theorem \ref{thm3.5} and is left to the reader.

\smallskip

We have the following corollaries. We formulate Corollary \ref{cor3.2}
concerning mapping classes of braids separately because it is simple
and useful, although it is a particular case of Corollary \ref{cor3.3}.
We call a braid $b$ irreducible if the associated mapping class $\mathfrak{m}_{b,\infty}$ is irreducible.

\medskip

\begin{cor}\label{cor3.2} Let $b \in {\mathcal B}_n$ be an
irreducible braid and let ${\mathfrak m}_b =\Theta_n(b) \in
{\mathfrak M} ({\mathbb D} ; \partial \, {\mathbb D} , E_n)$ be its
mapping class. Then
$$
h(\reallywidehat{{\mathfrak m}_b})
= h(\reallywidehat{{\mathfrak m}_b \cdot {\mathfrak m}_{\Delta_n^{2k}}})
=h(\reallywidehat{{\mathfrak m}_{b,\infty}})
$$
for each integer $k$.
\end{cor}

\smallskip

Recall that ${\mathfrak m}_{b,\infty} = {\mathcal H}_{\infty}
({\mathfrak m}_b)$ (see section \ref{sec:2.2a}), $\Delta_n$
is
the Garside element 
in ${\mathcal B}_n$ and
by $\widehat {\mathfrak{m}}$ we denote the conjugacy class of a
mapping class $\mathfrak{m}$.

\medskip

\noindent {\bf Proof.} Identify the set of elements of ${\mathfrak
m}_b$ with the set of elements  ${\mathfrak m}'_b \subset
{\mathfrak m}_{b,\infty}$ which are equal to the identity outside
the unit disc $\overline{\mathbb D}$. 
The entropy of a class is the infimum of entropies of mappings in
the class. We obtain the inequality $h(\widehat{{\mathfrak m}_b})  =
h(\widehat{{\mathfrak m}'_b})  \geq h(\widehat{{\mathfrak
m}_{b,\infty}})$.

On the other hand Theorem \ref{thm3.5} assigns to each absolutely extremal
representative of $\widehat{{\mathfrak m}_{b,\infty}}$ a
representative of $\widehat{{\mathfrak m}_{b,\infty}}$ which has the
same entropy and equals the identity in a neighbourhood of infinity
and, thus represents $\widehat{{\mathfrak m}'_b}$. Hence
$h(\widehat{{\mathfrak m}_{b,\infty}}) \geq h(\widehat{{\mathfrak
m}'_b})= h(\widehat{{\mathfrak m}_b})$. The first equality of the
statement of the corollary is Lemma \ref{lemm3.11} and the fact that the mapping class corresponding to $\Delta_n^2$ is the Dehn twist about a curve that is homologous in $\mathbb{D}\setminus E_n$ to $\partial \mathbb{D}$. \hfill $\Box$

\medskip

\begin{cor}\label{cor3.3} Let $X$ be a bordered Riemann
surface, and let $E$ be a finite set of distinguished points in
${\rm Int} \, X$. Let ${\mathfrak m}$ be an irreducible relative
isotopy class of mappings, ${\mathfrak m} \in {\mathfrak M} (X ;
\partial X , E)$. 
Then for the mapping $\mathcal{H}_{\zeta}$ defined by equation \eqref{eq2.10} in Section \ref{sec2:3}
the following equalities hold
$$
h(\widehat{\mathfrak m}) = h(\widehat{{\mathfrak m} \cdot {\mathfrak
m}}_D) = h(\widehat{\mathcal{H}_{\zeta}\mathfrak m})
$$
where ${\mathfrak m}_D$ is the mapping class of  an arbitrary
product of powers of Dehn twists about simple closed curves which
are free homotopic to boundary curves of $X$.
\end{cor}
The proof follows along the same lines as the proof of
Corollary \ref{cor3.2}. It relies on Theorem \ref{thm3.6}.

\chapter[Conformal invariants. The Main Theorem]    {Conformal invariants of homotopy classes of curves. The Main theorem.}
\label{chapter4}
\setcounter{equation}{0}

In this chapter we define the entropy and the conformal module of conjugacy classes of braids. We formulate our Main Theorem, stating that the entropy of each conjugacy class of braids is inverse proportional with factor $\frac{\pi}{2}$ to its conformal module.
Here we prove this theorem in the case of irreducible braids. The general case will be proved in the next two chapters.

The Main Theorem allows to apply methods and known results concerning the more intensively studied entropy to problems whose solutions are based on the concept of the conformal module. Vice versa, for instance
methods of quasi-conformal mappings related to the concept of extremal length and conformal module can be applied to give upper and lower bounds for the entropy of $3$-braids, the bounds differing by universal multiplicative constants (see Chapter \ref{chapter3-braids}).

We also define the conformal module of conjugacy classes of elements of the fundamental group of complex manifolds, in particular, of the twice punctured complex plane.
A version of the definition
for elements of fundamental g roups,
not merely of their conjgacy classes, appears more effective for applications to quantitative versions of the Gromov-Oka Principle that will be given in later Chapters.
In this chapter we compute these quantities explicitly in a number of examples.

\section{Ahlfors' extremal length and conformal module}
\label{sec:4.1a}

Ahlfors
defined the extremal
length of a family of curves in the complex plane as follows.
Let $\Gamma$ be a family each \index{$\Gamma$}
member of which consists of the union of no more than countably
many connected locally rectifiable (open, half-open, or closed) arcs or connected closed curves (loops) in the complex plane. (We do not require that this union reparametrizes to a
single connected curve.  )
In this context we will call also the elements of $\Gamma$
"curves".
Ahlfors defined the extremal length of the family $\Gamma$
as follows. For a non-negative measurable function $\varrho$ in the complex plane
he defines $A(\varrho) = \int\!\!\!\int_{\mathbb{C}} \varrho^2$. For an element $\gamma \in \Gamma$ and such a function $\varrho$  he puts
$L_{\gamma}(\varrho)= \int_{\gamma} \varrho |dz|$, if
$\varrho$ is measurable on $\gamma$ with respect to arc length
and $L_{\gamma}(\varrho)=\infty$ otherwise. Put  $L(\varrho)=
\inf_{\gamma \in \Gamma}L_{\gamma}(\varrho)$.
The extremal length of the family $\Gamma$ is the following
value
$$
\lambda(\Gamma)=
\sup_{\varrho}\frac{L(\varrho)^2}{A(\varrho)},
$$
where the supremum is taken over all non-negative measurable
functions $\varrho$ for which $A(\varrho)$ is finite and does
not vanish.
\index{$\lambda(\Gamma)$}  \index{$\mathcal{M}(\Gamma)$} \index{extremal length ! of families of curves} \index{conformal module ! of families of curves} \index{extremal length ! of a rectangle} \index{extremal length ! of an annulus}

It is not hard to see from this definition that the extremal length
is invariant under conformal mappings (Theorem 3 in
\cite{A1}, Chapter I.D).

The conformal module $\mathcal{M}(\Gamma)$ of the family $\Gamma$ is defined to be
\begin{align*}
\mathcal{M}(\Gamma)=\frac{1}{\lambda(\Gamma)}\,.
\end{align*}

\noindent {\bf Example 1.}
Let $R$ be an open rectangle in the complex plane $\mathbb{C}$. Unless said otherwise
the considered rectangles will always have sides
parallel to the coordinate axes. Denote the length of the horizontal
sides of $R$
by $\sf b$ and the length of the vertical sides by
$\sf a$.
(For instance, we may consider $R= \{z=x+iy:0<x<\sf b,\,
0<y<\sf a\,\}$.) The extremal length of the rectangle is defined to be the extremal length of the family of connected open arcs in the rectangle that join the two horizontal sides. An open arc is said to join the two horizontal sides if its limit set consists of two connected components contained in the closures of different horizontal sides.
A small computation shows that the extremal length of $R= \{z=x+iy:0<x<\sf b,\,
0<y<\sf a\,\}$ equals $\lambda(R)=\frac{\sf a}{\sf
b}$ and the conformal module equals $m(R)= \frac{\sf b}{\sf a}$ (see Example 1 in \cite{A1}, Chapter I.D).
\index{${\sf a}$} \index{${\sf b}$}

For a conformal mapping $\omega:R\to U$ of the rectangle $R$ onto a domain
$U\subset \mathbb{C}$ the image $U$ is called a curvilinear rectangle, if $\omega$ extends to a continuous mapping on the closure $\bar R$, and the restriction to each (open) side of $R$ is a homeomorphism onto its image. The images of the vertical (horizontal, respectively) sides of $R$ are called the vertical (horizontal, respectively) curvilinear sides of the curvilinear rectangle $\omega(R)$. The extremal length of the curvilinear rectangle $U$
is the extremal length of the family of open arcs in $U$ that join the two horizontal curvilinear sides. Thus the extremal length of $U$
equals the extremal length of $R$. (See \cite{A1}, Chapter I.D).
\index{curvilinear rectangle}

\medskip

\noindent {\bf Example 2.} Ahlfors \cite{A1}, Chapter I.D, defined the extremal
length of an annulus  $A= \{z \in \mathbb{C}:\; r_1<|z|<r_2\}$ in the complex plane as
the extremal length of the family of closed
curves which are contained in the annulus and represent
the conjugacy class of the positively oriented generator of the
fundamental group of the annulus. A simple computation (Example 3 in \cite{A1}, Chapter I.D) shows that
the extremal length of an annulus $\,A= \{z \in \mathbb{C}:\; r_1<|z|<r_2\}\,\,$
is equal to $\lambda(A)= \frac{2\pi}{\,\log(\frac{r_2}{r_1})}$, and the
conformal module \index{conformal module ! of an
annulus} of this annulus
equals
$\,m(A)= \frac{1}{2\pi}\,\log(\frac{r_2}{r_1})\,.\, $ \index{$m(A)$} Two
annuli of finite conformal module are conformally equivalent iff
they have equal conformal module. If a manifold $\Omega$ is
conformally equivalent to an annulus $A$ in the complex plane, its conformal module is
defined to be $m(A)$. Recall that any domain in the complex plane
with fundamental group isomorphic to the group of integer numbers
$\mathbb{Z}$ is conformally equivalent to an annulus.

\smallskip

\noindent Example 2 in \cite{A1}, Chapter I.D, shows that the extremal length of the family of open arcs $\gamma$ in $A= \{z \in \mathbb{C}:\; r_1<|z|<r_2\}\,$ that join the two boundary circles is equal to $\frac{\log\frac{r_2}{r_1}}{2\pi}$.

\medskip
{\bf Example 3.} The following example is a generalization of Example 1.
Let $\Phi$ be a real $C^1$-function $\Phi$ on the closure of an open interval $J$, and let $\textsf{b}$ be a positive number. We consider the curvilinear rectangle $R_{J,\Phi, \textsf{b}}\stackrel{def}{=} \{x+iy \in \mathbb{C}: y \in J,\, x \in (\Phi(y), \Phi(y)+\textsf{b})\}$.
The following lemma holds.
\begin{lemm}\label{lemm216} Let $\Phi$ be a real $C^1$-function on the closure of an open interval $J$ and let $\textsf{b}$ be a positive number.
Denote by $\Gamma_{\Phi}$ the set of curves in the curvilinear rectangle $R_{J,\Phi, \textsf{b}}= \{x+iy \in \mathbb{C}: y \in J,\, x \in (\Phi(y),
\Phi(y)+\textsf{b})\}$ which join the two horizontal curvilinear sides.  Suppose the absolute value $|\Phi'|$ of the derivative of $\Phi$ is bounded by the constant $C$. Then
$$
\lambda(\Gamma_{\Phi}) \leq (1 + C^2) \lambda(\Gamma_0),
$$
where $ \Gamma_0$ is the family corresponding to the function $\Phi_0$ which is identically equal to zero.
\end{lemm}
\index{$R_{J,\Phi, \textsf{b}}$}
\noindent {\bf Proof}. The proof is similar to Example 1 in Chapter 1 of \cite{A1}, Chapter I.D. For any measurable function $\varrho$ on $\mathbb{C}$ and any  $x \in (0,\textsf{b})$ we have
$$
\int_J \varrho(x + \Phi(y) +i y) \sqrt{1 +\Phi'(y)^2}dy \geq L_{\Gamma_{\Phi}}(\varrho).
$$
Integrate over the interval $(0,\textsf{b})$ and apply Fubini's Theorem and H\"older's inequality. Using the bound for $|\Phi'|$ we obtain
$$
\Big(\int \int_{R_{\Phi, \textsf{b}}} dm_2 \int \int_{R_{\Phi, \textsf{b}}} \varrho^2 \cdot(1+ C^2) dm_2\Big)^{\frac{1}{2}}\geq \textsf{b} L_{\Gamma_{\Phi}}(\varrho).
$$
Denote by $|J|$ the length of the interval $J$. We obtain
$$
\textsf{b} |J| (1+ C^2) A(\varrho) \geq \textsf{b}^2 L_{\Gamma_{\Phi}}(\varrho)^2.
$$
Hence,
$$
\frac{L_{\Gamma_{\Phi}}(\varrho)^2}{A(\varrho)} \leq \frac{|J|}{\textsf{b}}\cdot(1+C^2) = \lambda(\Gamma_0) (1+C^2).
$$
Taking the supremum over all measurable functions $\varrho$ with finite non-vanishing integral we obtain
$$
\lambda(\Gamma_{\Phi}) \leq (1+C^2)  \lambda(\Gamma_0).
$$
The lemma is proved. \hfill $\Box$

\medskip

For later use we formulate three theorems of Ahlfors.

For two families  $\Gamma_1$ and $\Gamma_2$ as above the
following relation is introduced by Ahlfors:
$\Gamma_1 < \Gamma_2$ if each "curve"  $\gamma_2 \in
\Gamma_2$ contains a "curve" $\gamma_1 \in \Gamma_1$.

Suppose $\Gamma_1$ and $\Gamma_2$ are contained in disjoint measurable sets.
Ahlfors defines the sum $\Gamma_1+\Gamma_2$ of two such
families as follows. Consider a "curve" $\gamma_1 \in \Gamma_1$ and a "curve"
$\gamma_2 \in \Gamma_2$. The sum $\gamma_1 +\gamma_2$ is the '' curve'' consisting of the union of $\gamma_1$ and $\gamma_2$. The set  $\Gamma_1+\Gamma_2$ is the set consisting of all sums $\gamma_1 +\gamma_2$ for $\gamma_1 \in \Gamma_1$ and
$\gamma_2 \in \Gamma_2$.

The following theorems were proved by Ahlfors.

\medskip

\noindent {\bf Theorem A.}(\cite{A1}, Ch.1.D Theorem 2) {\it If
$\Gamma' < \Gamma$ then $\lambda(\Gamma') < \lambda(\Gamma)$.}

\medskip

\noindent {\bf Theorem B.}(\cite{A1}, Ch.1.D Theorem 4) {\it If
the families $\Gamma_j$ are contained in disjoint measurable
sets then
$\sum \lambda(\Gamma_j) \leq \lambda(\sum \Gamma_j)$.}

\medskip
\begin{cor}\label{cor4.0}
Let $R$ be a rectangle and let $\alpha$ be an open arc in $R$, that extends to a closed arc
with endpoints on different open vertical sides of $R$. The arc $\alpha$ cuts $R$ into two curvilinear rectangles $R_1$ and $R_2$. The following estimate holds:
\begin{align}\label{eq4.1a}
\lambda(R) \geq \lambda(R_1) +\lambda(R_2)\,.
\end{align}
Further, let  $A= \{z \in \mathbb{C}:\; r_1<|z|<r_2\}\,$ be an annulus in the complex plane and $\alpha$ an open arc in $A$ that extends to a closed arc in $\bar A$ with endpoints on different boundary circles. The arc cuts $A$ into a curvilinear rectangle $R$. The following inequality holds
\begin{align}\label{eq4.1b}
\lambda(A) \geq \lambda(R)\,.
\end{align}
\end{cor}

\noindent {\bf Proof.} To prove inequality \eqref{eq4.1a} we let $R_1$ be the curvilinear rectangle below of $\alpha$ and $R_2$ the curvilinear rectangle above of $\alpha$. Let $\Gamma_j,\, j=1,2,$ be the family of open arcs in $R_j$ that join the two open horizontal curvilinear sides of $R_j$, and let $\Gamma$ be the family of open arcs in $R$ that join the two open horizontal sides of $R$. Then
by Theorem A (see \cite{A1}, Ch.1 Theorem 2)
$\lambda(\Gamma)\geq \lambda(\Gamma_1+\Gamma_2)$. By Theorem B (see \cite{A1}, Ch.1 Theorem 4)
$\lambda(\Gamma_1+\Gamma_2)\geq \lambda(\Gamma_1)+\lambda(\Gamma_2)$. By Example 1  of Section \ref{sec:4.1a} the
equalities $\lambda(\Gamma)=\lambda(R)$, $\lambda(\Gamma_j)=\lambda(R_j)$, $j=1,2$ hold.

For the proof of inequality \eqref{eq4.1b} we let $\Gamma(A)$ be the family of curves in $A$ that
represent the positively oriented generator of the fundamental group of $A$.
By Example 2 of Section \ref{sec:4.1a} the extremal length
 $\lambda(A)$ is equal to the extremal length $\lambda(\Gamma(A)) $.
Further,  the set  $A \setminus \alpha$ is a curvilinear rectangle with horizontal sides being the strands of $\alpha$ that are reachable from $A \setminus \alpha$ by moving clockwise or counterclockwise, respectively. Its extremal length $\lambda(A \setminus \alpha)$
is the extremal length
$\lambda(\Gamma(A \setminus \alpha))$ in the sense of Ahlfors \cite{A1} of the family $\Gamma(A \setminus \alpha)$
of curves in the curvilinear rectangle $A \setminus \alpha$ that join the two horizontal sides of the curvilinear rectangle.
By Theorem A (see \cite{A1}, Ch.1 Theorem 2) the  inequality
\begin{align}\label{eqfin2}
\lambda(\Gamma(A \setminus \alpha)) \leq  \lambda(\Gamma(A))\,
\end{align}
holds.
\hfill $\Box$
\index{curve ! simple closed dividing}

\section{Another notion of extremal length.}
\label{sec:4.1b}
We will use another notion of the extremal length of families of curves. This notion is in the spirit of the definition of the Kobayashi metric and
applies also to arbitrary families of curves, that are
homotopy classes of curves contained in complex manifolds, maybe, different from the complex plane.
\begin{defn}\label{defn4.02}
Let $\mathcal{X} $ be a complex manifold, $X$ a domain in $\mathcal{X} $, and
let $\mathcal{E}_1$ and $\mathcal{E}_2$ be relatively closed subsets of $X$.
Let $\textsf{h}=_{\mathcal{E}_1}{\textsf{h}}_{\,\mathcal{E}_2}$ be a homotopy class of curves
in $X$
with initial point in $\mathcal{E}_1$ and terminal point in $\mathcal{E}_2$. If
$\mathcal{E}_1=\mathcal{E}_2$ we write $\textsf{h}_{\,\mathcal{E}_1}$ instead of
$\textsf{h}=_{\mathcal{E}_1}{\sf{h}}_{\,\mathcal{E}_2}$. A continuous mapping
$f$ from an open rectangle into $X$ which admits a continuous extension
to the
closure of the rectangle (denoted again by $f$) is said to represent $\textsf{h}$ if the lower
open
horizontal side is mapped to $\mathcal{E}_1$, the upper horizontal side is mapped
to $\mathcal{E}_2$
and the restriction of $f$ to the closure of each maximal vertical segment in the rectangle
represents $\textsf{h}$.
\end{defn}
\index{$_{\mathcal{E}_1}{\textsf{h}}_{\,\mathcal{E}_2}$}
The extremal length of homotopy classes of such curves is defined as follows.
\begin{defn}\label{def4.1} For a complex manifold $\mathcal{X}$, a domain $X$ in
$\mathcal{X}$,
two relatively closed connected subsets $\mathcal{E}_1$ and $\mathcal{E}_2$ of $X$ and a homotopy class $\textsf{h}=_{\mathcal{E}_1}\textsf{h}_{\mathcal{E}_2}$ of curves in $X$
with initial
point in $\mathcal{E}_1$ and terminal point in $\mathcal{E}_2$ the extremal length
$\Lambda(\textsf{h})$ is defined as
\begin{align}\label{4.x}
\Lambda(\textsf{h})= & \inf \{\lambda(R): R\, {\rm a \,rectangle\,
which\, admits\, a\, holomorphic\, mapping\, to} \nonumber \\
& X \,{\rm that\,
represents}
\; \textsf{h}\}\,.
\end{align}
The conformal module
$\mathcal{M}(\textsf{h})$ of the class $\sf h$ is defined as
\begin{align}\label{4.y}
\mathcal{M}(\textsf{h})=\frac{1}{\Lambda(\textsf{h})}= & \sup \{m(R): R\, {\rm a \,rectangle\,
which\, admits\, a\, holomorphic\,} \nonumber \\{\rm mapping\, to}
& X \,{\rm that\,
represents}
\; \textsf{h}\}\,.
\end{align}
\end{defn}
\index{$\Lambda(\textsf{h})$} \index{$\mathcal{M}(\textsf{h})$}

A similar definition can be given for free homotopy classes of curves in a complex manifold, in other words for curves representing a conjugacy class of elements of the fundamental group of the manifold.
\begin{defn}\label{defn4.03}
Let $X$ be a complex manifold, and $\hat{e}$ a conjugacy class of elements of the fundamental group of $X$. A continuous mapping from an annulus $A=\{z\in \mathbb{C}: \frac{1}{r}<|z|<r\}$ with $1<r\leq \infty$ is said to represent $\hat e$, if the restriction to the circle $\{|z|=1\}$ with positive orientation (and, hence, the restriction to each curve that is free homotopic to this circle) represents $\hat e$.
\end{defn}
The extremal length and the conformal module of free homotopy classes of curves is defined as follows.
\begin{defn}\label{def4.2}
Let $X$ be a complex manifold, and $\hat{e}$ a conjugacy class of elements of the fundamental group of $X$.
The extremal length $\Lambda(\hat{e})$ of $\hat e$ is defined as
\begin{align}\label{eq4.x'}
\Lambda(\hat{e})\stackrel{def}= & \inf\{\lambda(A):
 A \, {\rm an \, annulus\, that \, admits\, a\, holomorphic\, mapping\, to}  \nonumber\\
&  X \,{\rm that\, represents\,} \hat{e}\}\,.
\end{align}
The conformal module $\mathcal{M}(\hat{e})$ of $\hat e$ is defined as
\begin{align}\label{eq4.y'}
\mathcal{M}(\hat{e})\stackrel{def}=  \frac{1}{\Lambda(\hat{e})}=    & \sup\{m(A):
 A \, {\rm an \, annulus\, that \, admits\, a\, holomorphic\,}\nonumber\\ {\rm mapping\, to\;}
&  X \;{\rm that\, represents\;} \hat{e}\}\;.
\end{align}
\end{defn}
\index{$\Lambda(\hat{e})$} \index{$\mathcal{M}(\hat{e})$}

For a curve $\gamma$ we consider the curve $\gamma^{-1}$ obtained from $\gamma$ by inverting orientation. For a family of curves $\Gamma$ we let $\Gamma^{-1}$ be the family of curves $\gamma^{-1}$ with $\gamma\in \Gamma$.
It follows immediately from the definitions that $\Lambda({\sf h})= \Lambda({\sf h}^{-1})$ and  $\Lambda(\hat{e})= \Lambda(\widehat{e^{-1}})$. Indeed,
suppose a mapping $f:R\to X$ represents $\sf h$. Put $-R=\{-z:z\in R\}$ and $f_-(z)=f(-z)$. Then $f_-:-R\to X$ represents ${\sf h}^{-1}$ The argument for $\hat e$ is similar.

We will see in the sequel that the definition of extremal length in the sense of Ahlfors and the present definition
essentially coincide for the families of curves in the plane which we will consider in this work.
We need the following two lemmas.

\begin{lemm}\label{lemm5} Let $R$ and $R'$ be
rectangles with sides parallel to the axes. Suppose $S'$ is the
vertical strip bounded by the two vertical lines which are
prolongations of the vertical sides of the rectangle $R'$.
Let $f: R \to S'$ be a holomorphic mapping with continuous extension to the closure that takes the two horizontal sides of $R$ into different horizontal sides of $R'$. Then
$$
\lambda (R) \geq \lambda (R') \,.
$$
Equality holds if and only if the mapping is a surjective conformal
mapping from $R$ onto $R'$.
\end{lemm}

The following lemma concerns holomorphic mappings between annuli.
\begin{lemm}\label{lemm6} Let  $A$ and $A'$ be two annuli and let $f$
be a
holomorphic mapping from $A$ into $A'$ which induces an isomorphism on
fundamental
groups.
Then
$$
\lambda(A) \geq \lambda(A')\,.
$$
Equality holds if and only if the mapping is a conformal mapping from
$A$ onto
$A'$.
\end{lemm}

\bigskip
\noindent \textbf{Proof of Lemma \ref{lemm5}}.
Normalize the rectangles and the mapping
so that ${R}= \{x+iy: x \in (0,1), y \in (0,\textsf{a}) \}\, $ and
${R}'=
\{x+iy: x \in (0,1), y \in (0,\textsf{a}')\} \, $. Denote the
continuous extension of $f$ to the closure of $R$ again by $f$.

We may assume that  $f$ maps the upper side of
${R}$ to the upper side of ${R}'$ and the lower side of ${R}$ to the
lower side of ${R}'$. Put $u= \mbox{Re}\, f$ and $v=\mbox{Im}\, f$. Then
\begin{align}\label{eq5b}
\textsf{a}'  =  \int_0^1  \textsf{a}'\,dx &= \int_0^1 (v(x, \textsf{a})
-v(x,0))dx =
\int_0^1 dx \int _0 ^{\textsf{a}} dy \frac{\partial}{\partial y}v(x,y) \\
\nonumber
& = \int_0^{\textsf{a}} dy\int_0^1 dx\, \frac{\partial}{\partial x}u(x,y) =
\int_0^{\textsf{a}} dy
\,(u(1,y)-u(0,y))\\ \nonumber
& \leq \int_0^{\textsf{a}} 1\, dy =\textsf{a}.
\end{align}
 We used the Cauchy-Riemann
equations. To justify, for instance, the third equality we take for each $x \in (0,1)$ the limit of the equality $v(x, \textsf{a}-\varepsilon)
-v(x,\varepsilon)=\int _{\varepsilon} ^{\textsf{a}-\varepsilon} dy\frac{\partial}{\partial y}v(x,y)$ for $\varepsilon \to +0$ and use that $v$ is continuous on the closure of $R$.

The relation \eqref{eq5b} implies the inequality $\textsf{a}' \leq \textsf{a}$.

If $\textsf{a}' = \textsf{a}$ then $u(1,y)-u(0,y)=1$ for each $y \in
(0,a)$.
Hence, the left side of ${R}$ is mapped to the left side of ${R}'$ and
the right
side  of ${R}$ is mapped to the right side of ${R}'$. Since also the
lower side
of   ${R}$ is mapped to the lower side of ${R}'$ and the upper side of
${R}$ is
mapped to the upper side of ${R}'$ the image of the positively oriented
boundary
curve of  ${R}$ has index $1$ with respect to any point of ${R}'$ and
index $0$ with respect to each point in $\mathbb{C}\setminus \bar R$.
By
the
argument principle $f(R)=R'$ and $f$ takes each value in ${R}'$ exactly
once. Hence,
$f$ is a conformal map of $R$
onto $R'$.       \hfill $\Box$

\medskip
\noindent {\bf Proof of Lemma \ref{lemm6}.}
Assume the annuli $A$ and $A'$ have center $0$,
smaller radius $1$ and larger radius $r$ and $r'$, respectively. The
set $A \setminus (0,\infty)$ is conformally equivalent to the rectangle
${R}=\{\xi+i\eta: \xi
\in (0,\log r), \eta \in (0,2 \pi)\}$. The exponential function covers
the annulus $A'$ by the strip $S'=\{\xi+i\eta: \xi \in (0,\log r'),
\eta \in \mathbb{R}\}$. We
obtain a holomorphic mapping  $g=U+iV$ from ${R}$ to $S'$ for which either
$V(\xi, 2\pi)=V(\xi,0)+2\pi$ or $V(\xi, 2\pi)=V(\xi,0)-2\pi$ for $\xi \in (0,\log r)$. Assume without loss of generality that the first option holds. Then
\begin{align}\label{eq5c}
2 \pi \log r   &=   \int _0^{\log r} (V(\xi,2\pi)-V(\xi,0))d\xi =
\int_0^{\log r} d\xi \int _0^{2\pi} d\eta \frac{\partial}{\partial
\eta}V(\xi,\eta) \\ \nonumber
& = \int_0^{\log r} d\xi \int _0^{2\pi}d\eta \frac{\partial}{\partial
\xi}U(\xi,\eta)= \int_0^{2\pi} (U(\log r,\eta)- U(0,\eta))d\eta \\
\nonumber
&\leq \int_0^{2\pi} \log r' d\eta = 2\pi \log r'.
\end{align}

Equality $r=r'$ holds iff $f$ maps the bigger circle of $A$  to the
bigger circle of $A'$ and maps the smaller circle of $A$  to the
smaller circle of $A'$. Since the map $f$ induces an isomorphism of
fundamental groups an application of the argument principle shows that
$f$ is a conformal mapping of $A$ onto $A'$.
\hfill $\Box$

\medskip

\noindent {\bf Example 1.} Let $0<{\sf a}<{\sf b}$ be two positive numbers.
\begin{figure}[h]
\begin{center}
\includegraphics[width=8cm]{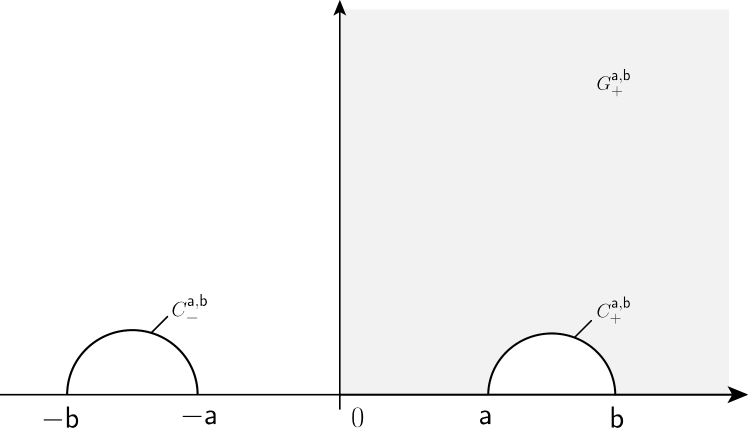}
\end{center}
\caption{The domains $G^{{\sf a},{\sf b}}$ and $G_{\pm}^{{\sf a},{\sf b}}$.}
\label{fig4.2}
\end{figure}
Denote by  $\Gamma^{{\sf a},{\sf b}}$ the family of curves in the upper half-plane $\mathbb{C}_+$ that join the two half-circles $C_{\pm}^{{\sf a}\,,{\sf b}}=\{z\in \mathbb{C}_+: |z\mp\frac{{\sf a}+{\sf b}}{2}|= \frac{{\sf b}- {\sf a}}{2}\}$ with diameter $({\sf a},{\sf b})$ and $(-{\sf b},-{\sf a})$, respectively. Similarly, let $\Gamma^{{\sf a},{\sf b}}_+$ be the family of curves in the upper half-plane $\mathbb{C}_+$ that join $C^{{\sf a},{\sf b}}_+$ with the imaginary half-axis.
\index{$\Gamma^{{\sf a},{\sf b}}$}
\index{$C^{{\sf a},{\sf b}}_+$} \index{$\Gamma^{{\sf a},{\sf b}}_+$}
\index{$\Gamma^{{\sf a},{\sf b}}_-$}\index{$C^{{\sf a},{\sf b}}_-$}
\index{$G^{{\sf a},{\sf b}}_+$}
\index{$G^{{\sf a},{\sf b}}_-$}

Let $G^{{\sf a},{\sf b}}_+$ be the domain
$G^{{\sf a},{\sf b}}_+=\{z\in \mathbb{C}: {\rm Re} z>0, {\rm Im} z>0, |z-\frac{{\sf a}+{\sf b}}{2}|> {\sf b}-{\sf a}\}$, and $G^{{\sf a},{\sf b}}= \{z\in \mathbb{C}_+:
|z \pm \frac{{\sf a}+{\sf b}}{2}|>  \frac{{\sf b}- {\sf a}}{2}\}$.
We consider $G^{{\sf a},{\sf b}}$  ($G^{{\sf a},{\sf b}}_+$, respectively) as curvilinear rectangles with curvilinear horizontal sides equal to $C^{{\sf a},{\sf b}}_{\pm}$ (equal to  $C^{{\sf a},{\sf b}}_+$ and
$\{z\in \mathbb{C}_+: {\rm Re}z>0\}$, respectively).
Notice, that the elements of $\Gamma_+^{{\sf a},{\sf b}}$ are not supposed to be contained in $G_+^{{\sf a},{\sf b}}$, and the respective remark concerns $\Gamma^{{\sf a},{\sf b}}$
and $\Gamma^{{\sf a},{\sf b}}_-$.

The following lemma holds.
\begin{lemm}\label{lem4.1'}
\begin{align}\label{eq4.x''}
\lambda(G^{{\sf a},{\sf b}}_{\pm})= \frac{1}{\pi}\log\frac{(\sqrt{\sf  a}+\sqrt{\sf b})^2}{{\sf b}-{\sf a}}\,,\;\;\quad
\lambda(G^{{\sf a},{\sf b}})= \frac{2}{\pi}\log\frac{(\sqrt{\sf a}+\sqrt{\sf b})^2}{{\sf b}-{\sf a}}\,.
\end{align}
Moreover,
\begin{align}\label{eq4.x'''}
\Lambda(\Gamma^{{\sf a},{\sf b}}_{\pm})= \frac{1}{\pi}\log\frac{(\sqrt{\sf a}+\sqrt{\sf b})^2}{{\sf b}-{\sf a}}\,,\;\;\quad
\Lambda(\Gamma^{{\sf a},{\sf b}})= \frac{2}{\pi}\log\frac{(\sqrt{\sf a}+\sqrt{\sf b})^2}{{\sf b}-{\sf a}}\,.
\end{align}
\end{lemm}
Recall that $\Lambda(\Gamma^{{\sf a},{\sf b}}_{\pm})$ ( $\Lambda(\Gamma^{{\sf a},{\sf b}})$, respectively)
denotes the extremal length of $\Gamma^{{\sf a},{\sf b}}_{\pm}$  ($\Gamma^{{\sf a},{\sf b}}$, respectively) in the sense of Definition \ref{def4.1}.
\index{$G^{{\sf a},{\sf b}}$} \index{$G^{{\sf a},{\sf b}}_+$}

\noindent{\bf Proof.}
\index{$T^{{\sf a},{\sf b}}$}
To prove equality \eqref{eq4.x''} we consider the
M\"obius transformation $T^{{\sf a},{\sf b}}$ that maps $0$ to $0$, $\infty$ to $1$, ${\sf a}$ to a positive real number $t< \frac{1}{2}$, and ${\sf b}$ to $1-t$. $T^{{\sf a},{\sf b}}$   is a conformal mapping of the Riemann sphere $\mathbb{P}^1$ onto itself, that takes $G^{{\sf a},{\sf b}}_+$ to a half-annulus and maps the horizontal curvilinear sides of  $G^{{\sf a},{\sf b}}_+$ to the half-circles. Each M\"obius transformation has the form $z\to \frac{a z +b}{ c z + d}$. The first two conditions for $T^{{\sf a},{\sf b}}$ imply that $b=0$ and $\frac{a}{c}=1$.
To determine the M\"obius transformation we ignore the normalization by the condition that the matrix with entries $a,b,c,d$ has determinant $1$ and put $a=c=1$. The remaining two conditions give the equations $\frac{\sf a}{{\sf a}+d}=t$, $\frac{\sf b}{{\sf b}+d}=1-t$, hence by eleminating $d$,
$$
{\sf a }(1-t)^2= {\sf b} t^2\,.
$$
We obtain the quadratic equation
$$
({\sf b}-{\sf a}) t^2 + 2 {\sf a}t -{\sf a}=0
$$
in $t$. Its solutions are
\begin{align}\label{eq4.1c}
t_{\pm}(=  t^{{\sf a},{\sf b}}_{\pm})= \frac{-{\sf a} \pm  \sqrt{{\sf ab}}}{{\sf b}-{\sf a}}\,.
\end{align}
The value $t_-$ is negative, so the required solution is $t_+<1-t_+$.
The half-annulus, whose boundary circles have center $\frac{1}{2}$
and radii $\frac{1}{2}- t_+=\frac{1}{2}- \frac{-{\sf a} +  \sqrt{{\sf ab}}}{{\sf b}-{\sf a}}      $
and $\frac{1}{2}$,  is a curvilinear rectangle which is the conformal image of a true rectangle with horizontal side length $\pi$ and vertical side length $\log\frac{1}{1-2 t_+}$. Hence,
$\lambda(G^{{\sf a},{\sf b}}_+)=\frac{1}{\pi}\log\frac{1}{1-2 t_+}\,.$
\index{$t_{\pm}$}
Since
\begin{align}\label{eq4.100}
\frac{1}{1-2 t_+}= \frac{1}{1- 2\frac{-{\sf a} + \sqrt{{\sf ab}}}{{\sf b}-{\sf a}}}=
\frac{{\sf b}-{\sf a}}{{\sf b}+{\sf a}-2\sqrt{{\sf ab}}}= \frac{(\sqrt{{\sf a}}+\sqrt{\sf{b}})^2}{{\sf b}-{\sf a}}\,,
\end{align}
the extremal length  $\lambda(G^{{\sf a},{\sf b}}_+) $
of the curvilinear rectangle $G^{{\sf a},{\sf b}}_+$ with horizontal sides being $C^{{\sf a},{\sf b}}_+=\{z\in \mathbb{C}_+: |z-\frac{{\sf a}+{\sf b}}{2}|= \frac{{\sf b}-{\sf a}}{2}\}$ and the imaginary half-axis is equal to $ \frac{1}{\pi} \log\frac{(\sqrt{{\sf a}}+\sqrt{{\sf b}})^2}{{\sf b}-{\sf a}}$.
Hence, we obtain equality \eqref{eq4.x''} for $\lambda(G_+^{{\sf a},{\sf b}})$.
Equalities \eqref{eq4.x''} for $\lambda(G_{-} ^{{\sf a},{\sf b}})$  and for $\lambda(G ^{{\sf a},{\sf b}})$ are obtained by the same reasoning.

For the proof of equality \eqref{eq4.x'''} we
consider the conformal mapping $R\to G^{{\sf a},{\sf b}}$ of a true rectangle $R$ onto $G^{{\sf a},{\sf b}}$ that maps the lower horizontal side of $R$ onto  $C^{{\sf a}, {\sf b}}_-$ and upper side to
$C^{{\sf a}, {\sf b}}_+$.
Apply Schwarz Reflection Principle to each horizontal side of $R$. We obtain
a conformal mapping from a true rectangle, denoted by $3R$, that has the same horizontal side length as $R$ and three times the vertical side length of $R$, onto the upper half-plane with two smaller half-discs removed. The removed half-discs are symmetric with respect to the imaginary axis.

After repeating $n$ times the application of the reflection principle we obtain a
conformal mapping from a rectangle $ 3^{n} \,R$ onto a domain $\Omega_n$
which is equal to the left half-plane with two half-discs removed. The
half-discs are symmetric with respect to the imaginary axis. The domains
$\Omega_n$ are increasing.  The diameter of the removed half-discs at step $n$ tends to zero for $n \to
\infty$ since the extremal length of $ 3^{n} \,R$ tends to $\infty$. We
obtain a conformal mapping, denoted by
$\mathfrak{c}$, of an
infinite strip $S'$ onto the left half-plane.

Let $f$ be any holomorphic mapping from a rectangle to the left
half-plane whose extension to the closure takes the upper side to
$C^{{\sf a}, {\sf b}}_+$ and the lower side to $C^{{\sf a}, {\sf b}}_-$ . Apply  Lemma \ref{lemm5} to the mapping $\mathfrak{c}^{-1} \circ f$.
Equality \eqref{eq4.x'''} for $\Lambda(\Gamma^{{\sf a},{\sf b}})$ follows. The inequality for
$\Lambda(\Gamma_{\pm}^{{\sf a},{\sf b}})$ is proved in the same way.
\hfill  $\Box$

\medskip

\noindent {\bf The fundamental group and the relative fundamental group of the twice punctured complex plane.}
The fundamental group $\pi_1(\mathbb{C}\setminus\{0,1\},\frac{1}{2})$ of the twice punctured complex plane with base point is a free group in two generators $a_1$ and $a_2$, were $a_1$ is represented by a curve that surrounds $0$ counterclockwise, and $a_2$ is represented by a curve that surrounds $1$ counterclockwise. The  fundamental group $\pi_1(\mathbb{C}\setminus\{0,1\},\frac{1}{2})$ is canonically isomorphic to the relative fundamental group $\pi_1(\mathbb{C}\setminus\{0,1\},(0,1))$. The elements of this group are represented by arcs with initial point and terminal point in the interval $(0,1)$. (The initial point may differ from the terminal point.) The isomorphism assigns to the element $e\in \pi_1(\mathbb{C}\setminus\{0,1\},\frac{1}{2})$ represented by a  curve $\gamma$ with base point $\frac{1}{2}$ the element in the relative fundamental group represented by $\gamma$.
For an element $e\in\pi_1(\mathbb{C}\setminus\{0,1\},\frac{1}{2})$ we denote be $e_{(0,1)}$ the image of $e$ under the mentioned canonical isomorphism.
\index{group ! relative fundamental group} \index{$\pi_1(\mathbb{C}\setminus\{0,1\},(0,1))$} \index{$e_{(0,1)}$}
\medskip

\begin{figure}[h]
\begin{center}
\includegraphics[width=11.5cm]{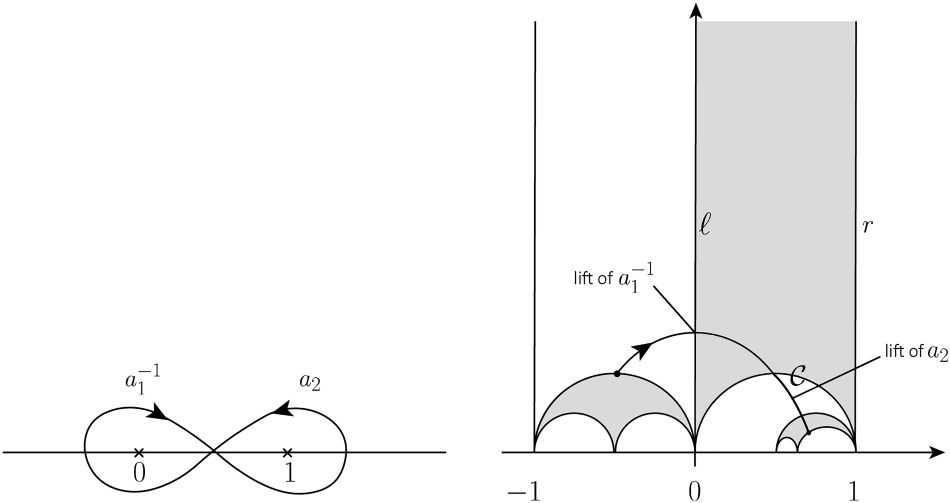}
\end{center}
\caption{A lift of $a_1^{-1}a_2$ to the universal covering.}
\label{fig4.4*}
\end{figure}

\noindent {\bf Example 2.} {\it The extremal length of $(a_1^k)_{(0,1)}$ and $(a_2^k)_{(0,1)}$
for  $k\in \mathbb{Z}$  equals zero, and the extremal length of $ (a_1^{-1}a_2)_{(0,1)}$ equals $\frac{2}{\pi}\log(2+\sqrt{5})$.}

This can be seen as follows. The class $(a_1)_{(0,1)}$ can be represented by the
holomorphic mapping $z\to e^z$ from  the rectangle $R=\{z=x+iy: x\in(-\infty,0), y\in (0,2\pi)\},$ of extremal length $\lambda(R)=0$ into $\mathbb{C}\setminus \{0,1\}$. There is a similar representation for all powers of $(a_1)_{(0,1)}$
and for all powers of $(a_2)_{(0,1)}$,
showing that the extremal length of these
elements of the relative fundamental group equals zero.

A representing arc for $(a_2)_{(0,1)}$ lifts  under the universal covering map ${\sf P}:\mathbb{C}_+ \to \mathbb{C}\setminus \{0,1\}$  to an arc in $\mathbb{C}_+$ that joins the half-circle with diameter $(0,1)$ with the half-circle with diameter$(\frac{2}{3}, 1)$ (see Section \ref{sec:2.0}). A representing arc for $(a_1)_{(0,1)}$ lifts under the universal covering map ${\sf P}:\mathbb{C}_+ \to \mathbb{C}\setminus \{0,1\}$ to an arc in $\mathbb{C}_+$ that joins the half-circle with diameter $(0,1)$ with the half-circle with diameter $(-1,0)$.
Hence, a representing arc for $(a_1^{-1}a_2)_{(0,1)}$
lifts to an arc in $\mathbb{C}_+$ that joins the half-circle with diameter $(-1,0)$
with the half-circle with diameter $(\frac{2}{3}, 1)$. (See Figure \ref{fig4.4*}.)
Any holomorphic mapping $f$ from a rectangle $R$ into $\mathbb{C}\setminus \{0,1\}$, that represents $(a_1^{-1}a_2)_{(0,1)}$, lifts to a holomorphic mapping from $R$ into $\mathbb{C}_+$ that takes the horizontal sides to the mentioned circles. The M\"obius transformation $z\to \frac{z}{z+1}$ maps $\mathbb{C}_+$ onto itself and takes $-1$ to $\infty$ and $0$ to $0$. Moreover it maps $\frac{2}{3}$ to $\frac{2}{5}$ and $1$ to $\frac{1}{2}$. Compose this mapping with the mapping $z\to 10 z$. The composition $\mathfrak{c}$ is a
conformal self-mapping of $\mathbb{C}_+$ that takes $-1,\,0,\, \frac{2}{3},\, 1$ to $\infty,\, 0,\, 4,\, 5$. Composing the lift of $f$ to $\mathbb{C}_+$ with this conformal self-mapping $\mathfrak{c}$ of $\mathbb{C}_+$, we obtain a holomorphic mapping $\mathfrak{c}\circ f$  from $R$ into $\mathbb{C}_+$ that takes the lower horizontal side to the imaginary half-axis and the upper horizontal side to the circle with diameter $(4,5)$. Vice versa, for each holomorphic mapping $g$ from a rectangle into $\mathbb{C}_+$, that takes the lower horizontal side into the imaginary axis and the upper horizontal side into the circle with diameter $(4,5)$, the composition $\mathfrak{c}^{-1}\circ g$
is a holomorphic mapping of a rectangle into the twice punctured complex plane that represents $(a_1^{-1}a_2)_{(0,1)}$. Lemma
\ref{lem4.1'} implies the equality $\Lambda((a_1^{-1}a_2)_{(0,1)})=\lambda(G^{4,5}_+)= \frac{1}{\pi}\log(2+\sqrt{5})^2=\frac{2}{\pi}\log(2+\sqrt{5})$.

\medskip
\noindent {\bf Example 3.} {\it The extremal length of the free homotopy class of closed curves $\widehat{a_1^{-1}a_2}$  equals $\frac{2}{\pi}\log(3+2\sqrt{2})$.}\\
This can be seen as follows.
We lift
a representative $\gamma$ of
$a_2\in\pi_1(\mathbb{C}\setminus\{0,1\},\frac{1}{2})$ to a curve in $\mathbb{C}_+$ with initial point $q$ on
the half-circle with diameter $(0,1)$. The lift
joins the half-circle with diameter $(0,1)$ with the half-circle with diameter $(\frac{2}{3},1)$.
There is a unique covering transformation $T_{a_2}$ that maps the initial point of this lift to its terminal point.

To obtain an explicit expression for the mapping $T_{a_2}$,
we make a change of the base point along a curve $\alpha$ in $\mathbb{C}\setminus \{0,1\}$. More precisely,
suppose $\alpha:[0,1]\to \mathbb{C}\setminus \{0,1\}$ is an arc with initial point $\frac{1}{2}$, 
such that $\alpha((0,1])$ is contained in $\mathbb{C}_+$. For $t\in [0,1]$ we
let $\alpha_t$ be the curves obtained by moving along $\alpha$ from the initial point $\alpha(0)=\frac{1}{2}$ to the point $\alpha(t)$. Consider the continuous family ${\rm Is}_{\alpha_t}(\gamma), \, t\in [0,1],$ of representatives of the free homotopy class $\reallywidehat{a_2}$, and their lift to the universal covering with initial point in $D_0$ if $t\in (0,1]$. The terminal point of each such lift is contained in the geodesic triangle $D_2$ with vertices $\frac{2}{3},1,\frac{1}{2}$ (see Section \ref{sec:2.0}). The set of covering transformations is discrete, and the family of covering transformations, that map the initial point of the considered lift of ${\rm Is}_{\alpha_t}(\gamma)$ to its terminal point, depends continuously on the parameter. Hence, each transformation of the family equals $T_{a_2}$. Hence, $T_{a_2}$ maps $D_0$ into $D_2$ and maps the half-circle with diameter $(0,1)$ to the half-circle with diameter $(\frac{2}{3},1)$. Consider in the same way the mapping $T_{a_2^{-1}}$, that is associated to the inverse $a_2^{-1}$, and its lift with initial point in $D_2$.
We see that the covering transformation $T_{a_2^{-1}}$ maps the geodesic triangle $D_2$ with vertices $\frac{2}{3},1,\frac{1}{2}$ into $D_0$, and maps the half-circle with diameter $(\frac{2}{3},1)$ into the half-circle with diameter $(0,1)$. The composition $T_{a_2^{-1}}T_{a_2}$ is the identity. Hence,  $T_{a_2}$ maps $D_0$ conformally onto the geodesic triangle $D_2$ with vertices $\frac{2}{3},1,\frac{1}{2}$, and maps the half-circle with diameter $(0,1)$ onto the half-circle with diameter $(\frac{2}{3},1)$.
Hence, $T_{a_2}|D_0$ coincides with the double reflection.
This can be used for explicit computation of $T_{a_2}$. One can also use the fact that $T_{a_2}$ is a conformal self-mapping of $\mathbb{C}_+$, hence extends to a M\"obius transformation, that takes the vertex $0$ (adjacent to the sides $\ell$ and $\mathcal{C}$) to $\frac{2}{3}$, the vertex $\infty$ (adjacent to the sides $\ell$ and $r$) to $\frac{1}{2}$, and the vertex $1$ (adjacent to $r$ and $\mathcal{C}$) to $1$. This M\"obius transformation equals $T_{a_2}= \frac{z-2}{2z-3}$.
Its inverse equals $T_{a_2}^{-1}=T_{a_2^{-1}}= \frac{-3z+2}{-2z+1}$.

In the same way we will now associate to $a_1\in \pi_1(\mathbb{C}\setminus \{0,1\},\frac{1}{2})$ a covering transformation $T_{a_1}$. We lift a representative of $a_1$ to a curve with initial point in the half-circle with diameter $(0,1)$. This lift joins the half-circle with diameter $(0,1)$ with
the half-circle with diameter $(-1,0)$. The covering transformation $T_{a_1}$ that takes the initial point of the lift to its terminal point, maps $D_0$ conformally onto the geodesic triangle with vertices $-1,-\frac{1}{2}, 0$, and maps the half-circle with diameter $(0,1)$ to
the half-circle with diameter $(-1,0)$.
Hence, $T_{a_1}$ maps $0$ to $0$, $1$ to $-1$, and $\infty$  to $-\frac{1}{2}$. A small computation gives $T_{a_1}=\frac{z}{1-2z}$. The inverse mapping equals $T_{a_1}^{-1}=T_{a_1^{-1}}=\frac{z}{2z+1}$.

The covering transformation corresponding to $a_1^{-1}a_2$ equals $T_{a_1^{-1}a_2    }(z)=T_{a_2}\circ T_{a_1}^{-1}(z) = \frac{3z+2}{4z+3} $. The matrix
$A=\begin{pmatrix}
3 & 2  \\
4 & 3
\end{pmatrix}$
has two real eigenvalues $t_{\pm}=3\pm 2\sqrt{2}$ whose product equals the determinant $1$.
The matrix $A$ can be conjugated to a diagonal matrix $B^{-1}AB= \begin{pmatrix}
t_+ & 0\\
0   & t_-
\end{pmatrix}$
by a matrix $B$ with real entries whose columns are eigenvectors. The M\"obius transformation corresponding to $B$ maps the real axis to itself. Changing if needed the direction of one of the eigenvectors we may assume that the M\"obius transformation maps $\mathbb{C}_+$ to itself. We see that the mapping $T_{a_1^{-1}a_2}:\mathbb{C}_+ \toitself$ can be conjugated by a conformal self-mapping of $\mathbb{C}_+$ to the mapping $T_{\lambda}, \, T_{\lambda}(z)= \frac{t_+ z}{t_-} =t_+^2 z= (3+2\sqrt{2})^2 z$. The quotient $\mathbb{C}_+\diagup {\langle T_{a_1^{-1}a_2} \rangle}$ is conformally equivalent to the quotient $\mathbb{C}_+\diagup T_{\lambda}$. This quotient is an annulus
of extremal length $\frac{1}{\pi}\log (3+2\sqrt{2})^2= \frac{2}{\pi}\log (3+2\sqrt{2})$.

Consider a holomorphic mapping $f$ of an annulus $A= \{z\in \mathbb{C}: r<|z|<R\}$ into $\mathbb{C}\setminus \{0,1\}$ that represents the free homotopy class $\widehat{a_1^{-1}a_2}$. Cut the annulus along the arc $\ell= (r,R)$, and lift the restriction $f|A\setminus \ell$ to a mapping $\tilde f$ into $\mathbb{C}_+$.
Let $\ell_-$ and $\ell_+$ be the strands of $\ell$ that are accessible from a point in $A\setminus \ell$ by moving clockwise, or counterclockwise, respectively, and for each point $p\in \ell$ we let $p_{\pm} \in \ell_{\pm}$ be the points corresponding to $p$. Then the equality $\tilde f (p_+) =T_{a_1^{-1}a_2}( \tilde{f}(p_-)) $ holds for the continuous extension of $\tilde f$ to $\ell_{\pm}$. This means, that the mapping $\tilde f$ from  $(A\setminus \ell)\cup \ell_-\cup \ell_+ $ descends to a holomorphic mapping from $A$ into
$\mathbb{C}_+\diagup T_{\lambda}$ that represents $\widehat{a_1^{-1}a_2}$. By Lemma \ref{lemm6}
the extremal length of $A$ is not smaller than $\lambda(\mathbb{C}_+\diagup T_{\lambda})=\frac{2}{\pi}\log (3+2\sqrt{2})$, and there is a conformal mapping of an annulus $A$ of extremal length $\frac{2}{\pi}\log (3+2\sqrt{2})$ onto $\mathbb{C}_+\diagup T_{\lambda}$ that represents $\widehat{a_1^{-1}a_2}$.

\bigskip

\noindent {\bf The extremal length of conjugacy classes of braids and
the extremal length of braids with totally real horizontal boundary values.}

These are conformal invariants of braids and  conjugacy classes of braids which will play a key role later.

Recall that for a subset $A$ of the
complex plane
$\mathbb{C}$ we defined the
configuration space $C_n (A) = \{(z_1 , \ldots , z_n) \in A^n
: z_i
\ne z_j$ for $i \ne j\}$ of $n$ particles moving along $A$
without collision.
Each
permutation in the symmetric group ${\mathcal S}_n$ acts on $C_n (A)$ by permuting
the
coordinates.
The quotient $C_n (A) \diagup {\mathcal S}_n$ is
called
the symmetrized configuration space related to $A$.
Recall that the natural projection $C_n (\mathbb{C}) \to  C_n (\mathbb{C}) \diagup {\mathcal S}_n$ is denoted by $\mathcal{P}_{\rm sym}$.

Choose a base point $E_n \in C_n (\mathbb {R}) \diagup {\mathcal S}_n$.
Recall that braids on $n$ strands  ($n$-braids for short)
with base point $E_n$ are homotopy classes of loops with
base point $E_n$ in the symmetrized configuration space,
equivalently, they are elements of the fundamental group $\pi_1(C_n(\mathbb {C})\diagup{\mathcal S}_n, E_n)$ of the symmetrized configuration space
with base point $E_n$.

Recall that conjugacy classes of $n$-braids are free homotopy classes of loops in $C_n(\mathbb {C})\diagup{\mathcal S}_n\cong \mathfrak{P}_n$.
According to Definition \ref{def4.2} the extremal length $\Lambda(\hat
b)$ of a conjugacy class of $n$-braids $\hat b$ is defined as $ \Lambda(\hat b)= {\rm inf}_{A \in
\mathcal{A}}\,
\lambda(A),$ where
$\mathcal{A}$ denotes the set of all annuli which admit a
holomorphic mapping into
$C_n(\mathbb{C}) \diagup \mathcal{S}_n$ that represents $\hat b$.

\index{extremal length ! of conjugacy classes of braids} \index{extremal length ! of braids with totally real horizontal boundary values}

To define the extremal length of braids with totally real horizontal boundary values we consider the totally real subspace $C_n(\mathbb{R})\diagup \mathcal{S}_n$ of $C_n(\mathbb{C})\diagup \mathcal{S}_n$.
The totally real subspace $\mathcal{E}^n_{tr}\stackrel{def}{=}\,C_n (\mathbb {R})
\diagup{\mathcal S}_n\,$ of $\,C_n (\mathbb {C}) \diagup {\mathcal
S}_n\,$ is connected and simply connected. Indeed, the totally real
subspace
$C_n (\mathbb {R})$ of $C_n (\mathbb {C})$ is the union of the
connected components $\,\{(x_1, \ldots , x_n) \in
\mathbb{R}^n: \,
x_{\sigma(1)} < x_{\sigma(2)} < \ldots < x_{\sigma(n)}\}\,$
over all
permutations $\sigma \in {\mathcal S}_n$. Thus $C_n (\mathbb
{R})$
is invariant under the action of ${\mathcal S}_n$ and the
quotient
is homeomorphic to $\{(x_1, \ldots , x_n) \in \mathbb{R}^n:
x_1 <
x_2 < \ldots < x_n\}\,$. Hence the claim.

The fundamental group $\pi_1(\,C_n (\mathbb {C}) \diagup
{\mathcal
S}_n\,,\; E_n\,)$ is isomorphic to the relative fundamental
group
$\pi_1(\,C_n (\mathbb {C}) \diagup {\mathcal S}_n\,,\; C_n
(\mathbb
{R}) \diagup {\mathcal S}_n \,)$. The elements of the latter
group
are homotopy classes of arcs in $\,C_n (\mathbb {C}) \diagup
{\mathcal S}_n\,$ with endpoints in $\,C_n (\mathbb {R})
\diagup
{\mathcal S}_n\,$.

The isomorphism between the two groups is obtained as follows.
Since the fundamental groups with different base point
are isomorphic, we may assume that $E_n$ is contained
in the
totally real subspace $\, C_n (\mathbb {R}) \diagup {\mathcal
S}_n\,$. Each element of $\pi_1(\,C_n (\mathbb {C}) \diagup
{\mathcal S}_n\,,\; E_n\,)$ is a subset of an element of
$\pi_1(\,C_n (\mathbb {C}) \diagup {\mathcal S}_n\,,\; C_n
(\mathbb
{R}) \diagup {\mathcal S}_n \,)$. Vice versa, since $\,C_n
(\mathbb
{R}) \diagup {\mathcal S}_n\,$ is connected and $E_n$ is
contained in  $\,C_n (\mathbb{R}) \diagup {\mathcal S}_n\,$,  each class
in $\pi_1(\,C_n (\mathbb {C}) \diagup
{\mathcal S}_n\,,\; C_n (\mathbb {R}) \diagup {\mathcal S}_n\,
)$
contains a class in $\pi_1(\,C_n (\mathbb {C}) \diagup
{\mathcal
S}_n\,,\; E_n\,)$. Since $\,C_n (\mathbb {R}) \diagup
{\mathcal
S}_n\,$ is simply connected, each class in $\pi_1(\,C_n
(\mathbb
{C}) \diagup {\mathcal S}_n\,,\; C_n (\mathbb {R}) \diagup
{\mathcal
S}_n \,)$ contains no more than one class of $\pi_1(\,C_n
(\mathbb
{C}) \diagup {\mathcal S}_n\,,\; E_n\,)$. Indeed, if two loops
in
$\,C_n (\mathbb {C}) \diagup {\mathcal S}_n\,$ with base point
$E_n$
are homotopic as loops in $\,C_n (\mathbb {C}) \diagup
{\mathcal
S}_n\,$ with varying base point in $\,C_n (\mathbb {R})
\diagup
{\mathcal S}_n\,$ then they are homotopic as loops in $\,C_n
(\mathbb {C}) \diagup {\mathcal S}_n\,$ with fixed base point
$E_n$.

Let $b \in \mathcal{B}_n$ be a braid.
Denote its image in the relative fundamental group
$\pi_1(\,C_n
(\mathbb {C}) \diagup {\mathcal S}_n\,,\; C_n (\mathbb {R})
\diagup
{\mathcal S}_n\, )$ by $b_{tr}$.
\index{$\pi_1(\,C_n (\mathbb {C}) \diagup {\mathcal S}_n\,,\; C_n
(\mathbb
{R}) \diagup {\mathcal S}_n \,)$}
\index{$b_{tr}$}

We are now ready to define for any braid its  extremal length  with
totally real boundary values (and the conformal module  with totally
real boundary values, respectively).
\begin{defn}\label{def1} Let $b \in \mathcal{B}_n$ be an
$n$-braid. The extremal length $\Lambda_{tr}(b)$ with totally
real horizontal
boundary values is defined as
\begin{align}
\Lambda_{tr}(b)=& \inf \{\lambda(R): R\, \mbox{ a rectangle
which
admits a holomorphic mapping to} \nonumber \\
&C_n (\mathbb {C}) \diagup {\mathcal S}_n \,\mbox{ that
represents}\; b_{tr}\}\,.\nonumber
\end{align}
The conformal module $\mathcal{M}_{tr}(b)$ of $b$ with totally
real horizontal
boundary values, respectively, is defined as
\begin{align}
\mathcal{M}_{tr}(b)= &\sup \{m(R): R\, \mbox{ a rectangle which
admits a holomorphic mapping to}\; \nonumber \\
& C_n (\mathbb {C}) \diagup {\mathcal S}_n \,\mbox{ that
represents}
\; b_{tr}\}\,.\nonumber
\end{align}
\end{defn}
\index{$\Lambda_{tr}(b)$} \index{$\mathcal{M}_{tr}(b)$}
We will also use the notation $\Lambda(b_{tr})$ for $\Lambda_{tr}(b)$ and the notation
$\mathcal{M}(b_{tr})$ for  $\mathcal{M}_{tr}(b)$.
Note that the two invariants are
inverse to each other. Sometimes it is more convenient to work with the extremal length, some other times it is more appropriate to speak about the  conformal module.

Recall that $\Delta_n$ denotes the Garside element in the braid group
$\mathcal{B}_n$.
\begin{lemm}\label{lemm1} For each braid $b\in \mathcal{B}_n$ the equalities $\Lambda(\hat b)= \Lambda(\widehat{b \Delta_n^2})$ and
$\Lambda_{tr}(b)= \Lambda_{tr}(b \Delta_n)=\Lambda_{tr}(\Delta_n b)$ hold.
\end{lemm}
\noindent {\bf Proof.} 
Let $R=\{x+iy: x \in (0, 1),\, y \in (0, \textsf{a})\}$, and suppose a holomorphic mapping $f=\{f_1,f_2,\ldots,f_n\}:R \to \,C_n(\mathbb{C}^n)
\diagup \mathcal{S}_n\, $ represents $b_{tr}$.
The mapping
$\zeta \to e^{\frac{\pi}{\textsf{a}}
\zeta} f(\zeta)= e^{\frac{\pi}{\textsf{a}}
\zeta}\{ f_1(\zeta), f_2(\zeta),\ldots   , f_n(\zeta)\} $ is holomorphic on
$R$ and represents
$(b\, \Delta_n)_{tr}$. Since $\Lambda_{tr}(b^{-1})=\Lambda_{tr}(b)$ for each braid $b$, we obtain $\Lambda_{tr}(\Delta_n b)=\Lambda_{tr}(b)$.
The stated relation for $\Lambda(\hat{b})$ is proved in the same way.
\hfill $\Box$

\medskip

\noindent {\bf Example 4.} {\it The extremal length with totally real horizontal boundary values of $\sigma_1^{-1} \sigma_2$ equals $\Lambda((\sigma_1^{-1} \sigma_2)_{\rm tr}) =\frac{\log(3+2\sqrt{2})}{\pi}$.}\\
To prove this fact, we let $\gamma$ be an arc in $C_3(\mathbb{C})\diagup \mathcal{S}_3$ that represents the homotopy class  $(\sigma_1 ^{-1}\sigma_2)_{\rm tr}\in \pi_1(C_3(\mathbb{C})\diagup \mathcal{S}_3, C_3(\mathbb{R})\diagup \mathcal{S}_3)$. The initial point (and also the terminal point) of $\gamma$ is an unordered triple of distinct points, each contained in the real axis. Lift $\gamma$ to an arc $\tilde{\gamma}=(\tilde{\gamma}_1,\tilde{\gamma}_2,\tilde{\gamma}_3)$ in $C_3(\mathbb{C})$, $\mathcal{P}_{\rm sym}(\tilde{\gamma})=\gamma$,  so that the initial point of the lift $\tilde{\gamma}$ is an ordered triple $(x_1,x_2,x_3)$ with $x_2<x_1<x_3$.

We
assign to each point $(z_1,z_2,z_3)$ the cross ratio
\begin{align*}
M_{(z_1,z_2,z_3)}\stackrel{def}= (z_2,z_3;z_1,\infty)
=\frac{z_2-z_1}{z_3-z_1}\cdot \frac{z_3-\infty}{z_2-\infty}=\frac{z_2-z_1}{z_3-z_1}\,.
\end{align*}
It defines a holomorphic mapping
\begin{align*}
C_3(\mathbb{C})\ni (z_1,z_2,z_3)\to M_{(z_1,z_2,z_3)}\in \mathbb{C}\setminus\{0,1\}\,.
\end{align*}

For later use we notice the following. In  Section \ref{sec:2.3a} we assigned to each class in $C_n(\mathbb{C})\diagup \mathcal{A}$ the unique representative  with first two coordinates equal to $0$ and $\frac{1}{n}$. Here it will be convenient to consider the
representative $\mathcal{P}'_{\mathcal{A}, {\sf n}}(z_1,z_2,z_3)=(0,M_z,1)$ of
$ (z_1,z_2,z_3)\diagup \mathcal{A}$ whose first entry equals $0$, and whose last entry equals $1$. We will denote by $\mathcal{P}'_{\mathcal{T}}$ the mapping that assigns to each
point $\tau\in \mathcal{T}(0,4)$ the representative of
 $\mathcal{P}_{\mathcal{T}}(\tau)  \diagup \mathcal{A}$ of the mentioned form
(compare with Lemma \ref{lemm2.5}).

We
associate to $\tilde{\gamma}$ the curve $M_{\tilde{\gamma}},\,M_{\tilde{\gamma}}(t)=M_{(\tilde{\gamma}_1(t),
\tilde{\gamma}_2(t),\tilde{\gamma}_3(t))},\, t\in [0,1], $ in  $\mathbb{C}\setminus\{0,1\}$.
The initial point of this curve is in $(-\infty,0)$.
Figure \ref{fig4.5} shows the curve $M_{\tilde{\gamma}}$ for a curve $\gamma$ in $C_3(\mathbb{C})\diagup \mathcal{S}_3$ representing $(\sigma_1^{-1} \sigma_2)_{\rm tr}$. Lift the curve $M_{\tilde{\gamma}}$ under the mapping $ \mathcal{P}'_{\mathcal{T}}: \mathcal{T}(0,4)\to \mathbb{C}\setminus\{0,1\}$
 to a curve $\tilde{M}_{\tilde{\gamma}}$ in the universal covering
$
\mathbb{C}_+\cong \mathcal{T}(0,4)$ of
$$
\mathbb{C}\setminus\{0,1\}\cong \{0\}\times \mathbb{C}\setminus\{0,1\}\times \{1\}\cong C_3(\mathbb{C})\diagup\mathcal{A}\,.
$$
We choose the lift
with initial point on the imaginary half-axis, see Figure  \ref{fig4.5}. Its terminal point is contained in the half-circle $\mathbb{C}_+\cap \{|z-\frac{3}{4}|=\frac{1}{4}\}$.
\index{cross ratio} \index{$M_{(z_1,z_2,z_3)}$}
\index{$M_{\tilde{\gamma}}$}
Consider a holomorphic mapping $g$ of a rectangle to $C_3(\mathbb{C)} \diagup \mathcal{S}_3$ that represents $(\sigma_1^{-1} \sigma_2)_{\rm tr} $. Lift this mapping to a holomorphic mapping $\tilde g$ into $C_3(\mathbb{C)}$,
consider the associated mapping $M_{\tilde{g}}:R\to \mathbb{C}\setminus\{0,1\}$, and lift it to a mapping
$\tilde{M}_{\tilde{g}}:R\to \mathbb{C}_+$, so that $\tilde{M}_{\tilde{g}}$ represents the curve  $\tilde{M}_{\tilde{\gamma}}$  in Figure \ref{fig4.5}. The mapping $\tilde{M}_{\tilde{g}}$ is holomorphic and maps the lower side of $R$ to the imaginary half-axis, and the upper side to the half-circle
$\mathbb{C}_+\cap \{|z-\frac{3}{4}|=\frac{1}{4}\}$.
Since $\frac{(\sqrt{\frac{1}{2}} +\sqrt{1})^2}{\frac{1}{2}}=2(\frac{1}{2}+2\sqrt{\frac{1}{2}}+1)= 3+2\sqrt{2}$,
by Lemmas \ref{lemm5} and \ref{lem4.1'} the extremal length $\lambda(R)$ is not smaller than
$\frac{\log(3+2\sqrt{2})}{\pi}$.

\begin{figure}[h]
\begin{center}
\includegraphics[width=9.0cm]{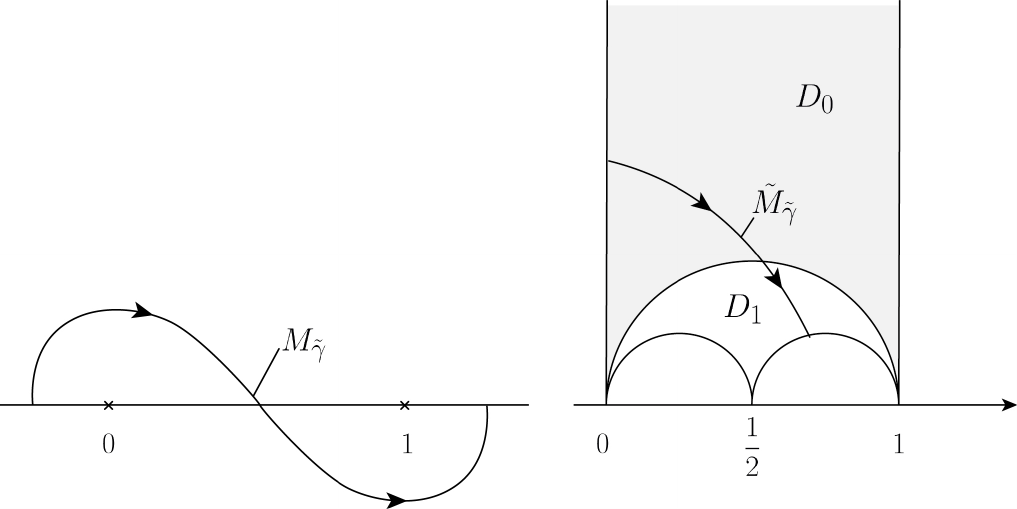}
\end{center}
\caption{The curve $\tilde{M}_{\tilde{\gamma}}$ on the Teichm\"uller space $\mathcal{T}(0,4)$, that corresponds to a representative of  $(\sigma_1^{-1} \sigma_2)_{\rm tr}$.}\label{fig4.5}
\end{figure}
Vice versa, a rectangle of extremal length $\frac{\log(3+2\sqrt{2})}{\pi}$ admits a conformal mapping $\mathfrak{c}$ onto $\{{\rm Re} z >0, {\rm Im} z >0, |z-\frac{3}{4}|>\frac{1}{4}\}$, such that the horizontal sides of $R$ correspond to the imaginary half-axis, and to the half-circle
$\mathbb{C}_+\cap \{|z-\frac{3}{4}|=\frac{1}{4}\}$. The composition $\mathcal{P}'_{\mathcal{T}}\circ \mathfrak{c}$ of $\mathfrak{c}$ with the universal covering map $\mathcal{P}'_{\mathcal{T}}$ is a holomorphic mapping from the rectangle into $\mathbb{C}\setminus \{0,1\} $ that represents the curve
 $\tilde{\gamma}$ (see Figure \ref{fig4.5}). The mapping $R\ni z \to (0, \mathcal{P}'_{\mathcal{T}}\circ \mathfrak{c}(z),1)\in C_3(\mathbb{C})$ projects to a mapping into $C_3(\mathbb{C})\diagup \mathcal{S}_3$ that represents $(\sigma_1^{-1} \sigma_2)_{\rm tr}$.

In the same way we see that the extremal length
$\Lambda((\sigma_1^{-1} \sigma_2^2 \sigma_1^{-1} )_{\rm tr})$ equals $\frac{2\log(3+2\sqrt{2})}{\pi}$. For a curve $\gamma$ representing $(\sigma_1^{-1} \sigma_2^2 \sigma_1^{-1})_{tr}$
the curve $\tilde{M}_{\tilde{\gamma}}$ joins the two half-circles $\{ |z\pm \frac{3}{4}|=\frac{1}{4}\}\cap \mathbb{C}_+$.

\medskip

\medskip

\noindent {\bf Example 5.} {\it The extremal length of the conjugacy class of $3$-braids $\reallywidehat{\sigma_1^{-1}\,\sigma_2}$ equals $\Lambda(\reallywidehat{\sigma_1^{-1}\,\sigma_2})= \frac{2}{\pi}\log\frac{3+\sqrt{5}}{2}$. This is the smallest non-vanishing extremal length among conjugacy classes of $3$-braids.}

To prove these facts, we will first describe the action on the Teichm\"uller space $\mathcal{T}(0,4)\cong \mathbb{C}_+$ of the modular transformations $T_{\sigma_1}$ and $T_{\sigma_2}$ corresponding to $\sigma_1$ and $\sigma_2$. We will use the normalized mappings $\mathcal{P}'_{\mathcal{A},{\sf n}}: C_3(\mathbb{C})\to \mathbb{C}\setminus\{0,1\}$ and $\mathcal{P}'_{\mathcal{T}}:\mathcal{T}(0,4)\cong \mathbb{C}_+ \to \mathbb{C}\setminus\{0,1\}$ as in Example 4 of this section.

Recall that the modular transformation $T_b$ associated to an $n$-braid $b$ is the modular transformation $\varphi_b^*$ of a self-homeomorphism $\varphi_b$ of $\mathbb{P}^1$ that represents the mapping class $\mathfrak{m}_{b,\infty}$ corresponding to $b$.
We consider the modular transformation $T_b$ for the $3$-braids $b=\sigma_1$ and $b=\sigma_2$.
With the same normalization as in Example 4 of this section the mapping $\varphi_b$ fixes $\infty$ and maps the set $E_3=\{0,\frac{1}{2}, 1\}$ onto itself. Such a homeomorphism $\varphi_b$ can be obtained using a parameterizing isotopy $\varphi^t$ for a geometric braid $\gamma$
representing $b$, i.e. a continuous family of self-homeomorphisms of $\mathbb{P}^1$
such that $\varphi_t(E_3)=\gamma(t), \, t\in [0,1]$. We may take $\varphi_b=\varphi^1$ and $T_b=(\varphi^1)^*$.
\index{$\varphi_b$} \index{$\varphi_b^*$}

Lift the base point $E_3=\{0,\frac{1}{2},1\}\in C_3(\mathbb{C})\diagup \mathcal{S}_3$ under $\mathcal{P}_{\rm{sym}}$ to the point $\tilde{E}_3=(0,\frac{1}{2},1)$ of $C_3(\mathbb{C})$. We identify the point  $\mathcal{P}'_{\mathcal{A},{\sf n}}(\tilde{E}_3)$ with $\frac{1}{2}$. Take the preimage $\frac{1+i}{2}$ of $\frac{1}{2}$ under $\mathcal{P}'_{\mathcal{T}}$. The image of $\frac{1+i}{2}$ under $T_{\sigma_1}$ can be obtained as follows.

Consider a curve $\gamma_1$ in $C_3(\mathbb{C})\diagup \mathcal{S}_3$ that represents $\sigma_1$ with base point $E_3$ and its lift $\tilde{\gamma}_1$ under $\mathcal{P}_{\rm sym}$ with initial point $\tilde{E}_3$. The projection $\mathcal{P}'_{\mathcal{A},{\sf n}}(\tilde{\gamma}_1)$ (identified with the curve $t\to M_{\tilde{\gamma}_1(t)}$) is a curve in $\mathbb{C}\setminus\{0,1\}$ with initial point $\frac{1}{2}$.
The curve $\gamma_1$ may be chosen
so that  $\mathcal{P}'_{\mathcal{A},{\sf n}}(\tilde{\gamma}_1)$  joins $\frac{1}{2}$ with a point in the negative real axis, and the interior points of the curve  $\mathcal{P}'_{\mathcal{A},{\sf n}}(\tilde{\gamma}_1)$ are contained in upper half-plane $\mathbb{C}_+$.

Consider the lift $\tilde{M}_{\tilde{\gamma}_1}$ of $M_{\tilde{\gamma}_1}$  under $\mathcal{P}'_{\mathcal{T}}:\mathcal{T}(0,4)\cong \mathbb{C}_+ \to \mathbb{C}\setminus\{0,1\}$ with initial point $\frac{1+i}{2}$. The lift $\tilde{M}_{\tilde{\gamma}_1}$ of $M_{\tilde{\gamma}_1}$ joins $\frac{1+i}{2}$ with a point on the positive imaginary axis. The interior of $\tilde{M}_{\tilde{\gamma}_1}$ is contained in $D_0$.

The terminal point of  $\tilde{M}_{\tilde{\gamma}_1}$
is the image of $\frac{1+i}{2}$ under $T_{\sigma_1}$.
Indeed, the following commutative diagram (see also Lemma \ref{lemm2.5} )
$$
\xymatrix{
\varphi^t \ar[rr]^{[ \ ]} \ar[d]^{e_3} && {\tilde{\mathcal{M}}_{\tilde{\gamma}_t}}
\ar[d]^{{\mathcal P}'_{\mathcal T}} \\
\tilde{\gamma}_1(t) \ar[rr]^{{\mathcal P}'_{\mathcal{ A},{\sf n}}}
&& \;\;\mathcal{M}_{\tilde{\gamma}_t}}
$$
implies the equality
\begin{align*}
T_{\sigma_1}(\frac{1+i}{2})=   \varphi_{\sigma_1} ^*([{\rm Id}]) = [\varphi_{\sigma_1}]= [\varphi^1]=\tilde{M}_{\tilde{\gamma_1}(1)}\,.
\end{align*}
We proved that $T_{\sigma_1}(\frac{1+i}{2})$ is contained in the imaginary half-axis.

We will now prove that $T_{\sigma_1}$ maps the half-circle $C_+^{0,1}=\{|z-\frac{1}{2}|=\frac{1}{2}\}\cap \mathbb{C}_+$ into the positive imaginary half-axis. Take any point in $C_+^{0,1}$ and join it with $\frac{1+i}{2}$ by an arc $\alpha(s),\, s \in [0,1]$, in $C_+^{0,1}$.
Let $w_s, \, s\in [0,1],$ be a continuous family of self-homeomorphisms of $
\mathbb{P}^1$, that fix $\infty$, such that $w_0={\rm Id}$ and $\tau_s\stackrel{def}=[w_s]=\alpha(s)$.
(See also Lemma \ref{lem2.1a}.)
Notice that for $\tilde{E}_3^s\stackrel{def}=w_s(\tilde{E}_3)$ the point $x(s)\stackrel{def}=\mathcal{M}_{\tilde{E}_3^s}$
is contained in $(0,1)$.  The equality
\begin{align*}
T_{\sigma_1}(\tau_s)=(\varphi_{\sigma_1}^1)^*([w_s])=[w_s \circ \varphi_{\sigma_1}^1]\,
\end{align*}
holds. Since $\varphi_{\sigma_1}^1$ acts on $\tilde{E}_3$ by permuting the first two coordinates, the homeomorphism $w_s\circ\varphi_{\sigma_1}^1$ acts on
$w_s(\tilde{E}_n)=\tilde{E}_n^s$ by permuting the first two coordinates. This means that
$x'(s)\stackrel{def}=\mathcal{M}_{w_s\circ\varphi_{\sigma_1}^1(\tilde{E}_3)}\in (-\infty,0)\subset \mathbb{C}\setminus\{0,1\}$. Hence, the lift $[w_s\circ\varphi_{\sigma_1}^1]$ of $\mathcal{M}_{w_s\circ\varphi_{\sigma_1}^1(\tilde{E}_3)}$ under $\mathcal{P'}_{\mathcal{T}}$ is contained in the imaginary half-axis.

Each member of the continuous family of curves $\{\mathcal{M}_{w_s\circ\varphi_{\sigma_1}^t(\tilde{E}_3)},\, t\in [0,1]\}_{s\in [0,1]}$ in $\mathbb{C}\setminus \{0,1\}$ has initial point on $(0,1)$ and terminal point in $(-\infty,0)$. The family lifts to a continuous family of curves in $\mathbb{C}_+$ with initial point in the half-circle $C^{0,1}_+$ and terminal point in the imaginary half-axis, and the interior  $\{\mathcal{M}_{w_0\circ\varphi_{\sigma_1}^t(\tilde{E}_3)},\,t\in(0,1)\}$ of the curve corresponding to the parameter $s=0$, is contained in the upper half-plane.
Hence, $T_{\sigma_1}$ maps the half-circle $C^{0,1}_+$  into the positive imaginary half-axis.

The same argument for the braid $\sigma_1^{-1}$ shows that the modular transformation $T_{\sigma_1}$ maps the half-circle $\{z\in \mathbb{C}_+: |z-\frac{1}{2}|=1\}$ onto the positive imaginary half-axis.

\index{geodesic triangle}
We show now that  $T_{\sigma_1}$  maps the geodesic triangle $D_1$ with vertices $0,\frac{1}{2}, 1$ onto $D_0$. Let  $\alpha(s),\, s \in [0,1],$ be
an arc in $\mathbb{C}_+\cong \mathcal{T}(0,4)$ with $\alpha(0)=\frac{1+i}{2}$ and $\alpha((0,1])\subset D_1$. The projection $\mathcal{P}'_{\mathcal{T}}(\alpha((0,1]))$ to $\mathbb{C}\setminus\{0,1\}$ is contained in $\mathbb{C}_-$ and $\mathcal{P}'_{\mathcal{T}}(\alpha(0))= \mathcal{P}'_{\mathcal{T}}(\frac{1+i}{2})=\frac{1}{2}$. Let again $w_s$ be a continuous family of self-homeomorphisms of $\mathbb{P}^1$ for which $w_0={\rm Id}$ and $\tau_s\stackrel{def}=[w_s]=\alpha_s$.  Consider again the continuous family of curves $\{\mathcal{M}_{w_s\circ\varphi_{\sigma_t}^1(\tilde{E}_3^0)},\, s\in [0,1]\}$ in $\mathbb{C}\setminus \{0,1\}$ with initial points
$\{\mathcal{M}_{w_s(\tilde{E}_3^0)},\, s\in [0,1]\}$ in $(0,1)\cup \mathbb{C}_-$.
The terminal points  $\{\mathcal{M}_{w_s\circ\varphi_{\sigma_1}^1(\tilde{E}_3^0)},\, s\in [0,1]\}$
are contained in $(-\infty,0)\cup \mathbb{C}_+$, since $\varphi_{\sigma_1}^1$ acts on $\tilde{E}_3^0$ by permuting the first two coordinates.
Indeed, for a point $(z_1,z_2,z_3)\in C_3(\mathbb{C})$ the inclusion $M_{(z_1,z_2,z_3)}\in \mathbb{C}_-$ means, that
the point $z_2$ lies on the right of the line through $z_1$ and $z_3$, oriented in the direction from $z_1$ to $z_3$. Then $z_1$ lies on the left of the line through $z_2$ and $z_3$, oriented in the direction from $z_2$ to $z_3$.
This means, $M_{(z_2,z_1,z_3)}\in \mathbb{C}_+$.
The lift under $\mathcal{P}'_{\mathcal{T}}\cong {\sf P}$ of the
family of curves $\{\mathcal{M}_{w_s\circ\varphi_{\sigma_1}^t(\tilde{E}_3^0)},\, t\in [0,1]\}_{s\in [0,1]},$ in $\mathbb{C}\setminus \{0,1\}$ provides a continuous family of curves in $\mathbb{C}_+$ with initial points in
$D_1\cup C^{0,1}_+$ and terminal points in $D_0\cup \{{\rm Re}z=0\}\cap \mathbb{C}_+$. The terminating point of each curve is obtained from its initial point by applying the modular transformation $ T_{\sigma_1}$.
We showed that
$T_{\sigma_1}$ maps the geodesic triangle with vertices $0,\frac{1}{2},\infty$ into $D_0$. Similar arguments for the braid $\sigma_1^{-1}$ imply that $T_{\sigma_1}$ maps the geodesic triangle with vertices $0,\frac{1}{2},\infty$ onto $D_0$. Since $T_{\sigma_1}$ maps
the half-circle $\{z\in \mathbb{C}_+: |z-\frac{1}{2}|=1\}$ onto the positive imaginary half-axis, $T_{\sigma_1}$ takes $0$ to $0$, $1$ to $\infty$, and $\frac{1}{2}$ to $1$.
Hence, $T_{\sigma_1}(z)=\frac{z}{1-z}$ and $T_{\sigma_1^{-1}}(z)= \frac{z}{1+z}$.

The modular transformation $T_{\sigma_2}$ is computed similarly.
Represent $\sigma_2$ by a closed curve $\gamma_2:[0,1]\to C_3(\mathbb{C})\diagup \mathcal{S}_3$ with initial point $E_3^x=\{0,x,1\}$, where $x\in(0,1)$. We associate to it a curve $\tilde{M}_{({\tilde{\gamma}_2})_t}$
in $\mathcal{T}(0,4)\cong \mathbb{C}_+$ with initial point
contained in the half-circle $\{|z-\frac{1}{2}|=\frac{1}{2}\}$. The terminal point is contained in the half-circle $\{|z-\frac{3}{4}|=\frac{1}{4}\}$. By the same arguments as used for $\sigma_1$ we see that $T_{\sigma_2}$ is a M\"obius transformation that maps the half-circle $\{|z-\frac{1}{2}|=\frac{1}{2}\}$ onto the half-circle $\{|z-\frac{3}{4}|=\frac{1}{4}\}$, and
maps $D_0$ conformally onto the geodesic triangle with vertices $0,\frac{1}{2},1$.
Hence, $T_{\sigma_2}$   maps $0$ to $\frac{1}{2}$, $1$ to $1$,  and $\infty$ to $0$. This implies that the mapping equals $T_{\sigma_2}(z)=\frac{1}{-z+2}$.
(See also Figure \ref{fig4.5}.)

The modular transformation corresponding to $\sigma_1^{-1}\sigma_2$ equals $T_{\sigma_1^{-1}\sigma_2}=
T_{\sigma_2}\circ T_{\sigma_1^{-1}}(z)=\frac{z+1}{z+2}$. The eigenvalues of the matrix $\begin{pmatrix}
1 & 1  \\
1 & 2
\end{pmatrix}$ are $t_{\pm}=\frac{3\pm\sqrt{5}}{2}$.
The same arguments as in Example 3 of this section show that $\Lambda(\reallywidehat{\sigma_1^{-1}\sigma_2})=
\frac{2}{\pi}\log\frac{3+\sqrt{5}}{2}$.

For each $3$-braid $b$ the respective modular transformation $T_b$ is a finite product of powers of the $T_{\sigma_j}$. Hence, each $T_b$ is an element of $PSL_2(\mathbb{Z})$. If for the trace $a+d$ of the matrix corresponding to $T_b$ the equality $|a+d|=2$  holds, then $T_b$ is conjugate to a translation by a real number, and the quotient $\mathbb{C}_+\diagup \langle T_b\rangle$ is conformally equivalent to the once punctured plane,
hence has extremal length $0$. If $|a+d|<2$, then a power of the mapping is the identity. Hence, in a neighbourhood of the fix point the mapping is conjugate to the mapping $z\to e^{i \frac{2\pi}{ k}}\,z $. The quotient $\{0<|z|<\delta\}\diagup (z\sim e^{i \frac{2\pi }{ k}}\,z) $ is a punctured disc and has extremal length $0$. Hence, $\mathbb{C}_+\diagup T_b$ has extremal length $0$. If $|a+d|>2$ the extremal length of $\mathbb{C}_+\diagup T_b$ equals $\frac{2}{\pi}\log|t_+|$ for the larger absolute value $|t_+|= \frac{|a+d|}{2} +\frac{\sqrt{(a+d)^2-4}}{2}$   among the eigenvalues of the matrix. The smallest among these values over all matrices in $SL_2(\mathbb{Z})$ is obtained for $|a+d|=3$.  Hence,  $\frac{2}{\pi}\log\frac{3+\sqrt{5}}{2}$ is the smallest non-vanishing extremal length among conjugacy classes of $3$-braids. \hfill $\Box$

\medskip

Notice that the trace of the matrix corresponding to $T_{\sigma_1}$ and to $T_{\sigma_2}$ equals $2$, and both eigenvalues are equal to $1$. Hence, the matrix is conjugate by a matrix in $SL_2(\mathbb{R})$ to the sum of the unit matrix and an upper diagonal matrix. Hence, both eigenvalues of all non-trivial powers of $T_{\sigma_1}$ and of all non-trivial powers of
$T_{\sigma_2}$ are equal to $1$. This implies that $\Lambda(\widehat{\sigma_1^k})=\Lambda(\widehat{\sigma_2^k})=0$ for all integers $k\neq 0$. Moreover, for $T_{\sigma_1 \sigma_2}(z)=T_{\sigma_2}\circ T_{\sigma_1}(z)=\frac{z-1}{3z-2}$ the associated matrix has trace $1$. The mapping has a fix point in $\mathbb{C}_+$. The eigenvalues of the associated matrix are $\frac{1}{2}\pm\frac{\sqrt{3}}{2}i$ (see Section \ref{sec:2.0}). The power $T_{\sigma_1 \sigma_2}^3$ is equal to the identity and $\Lambda(\reallywidehat{\sigma_1 \sigma_2})=0$.

\section [Invariants of conjugacy classes of braids. Main Theorem.]{Invariants of conjugacy classes of braids. Statement of the Main Theorem.}
\label{sec:4.0}

\noindent{\bf The entropy of braids.}
Recall that the entropy $h(b)$ of a braid $b\in \mathcal{B}_n$ is defined as the infimum of the entropies of self-homeomorphims of the closed disc $\overline{\mathbb{D}}$ that are contained in the mapping class $\mathfrak{m}_b$ associated to $b$,

\begin{equation}\label{eq4.01}
h(b)=\inf\{h(\varphi):\varphi\in \mathfrak{m}_b\}
\end{equation}

The value $h(\varphi)$ is invariant under conjugation with self-homeomorphisms of the closed disc $\overline{\mathbb{D}}$, hence it does not depend on the position of the set of distinguished points and on the choice of the representative of the conjugacy class $\hat b$. We write $h(\hat b)=h(b)$. In Chapter \ref{chapter-entropy} we studied the entropy of irreducible mapping classes and braids.
\index{$h(b)$} \index{$h(\hat{b})$}

The entropy of a mapping class
is an important dynamical invariant which measures the
complexity of its behaviour in terms of iterations. It has received a lot of attention and has been studied intensively. As for braids (and the associated mapping classes), it has been known
that the entropy of any irreducible braid is the logarithm of
an algebraic number. Further, the lowest non-vanishing entropy $h_n$
among irreducible braids on $n$ strands, $n \geq 3,\,$ has been
estimated from below by $\frac{\log 2}{4} n^{-1}$ (\cite{P}) and has
been computed for small $n$. The entropy of a few more  braids has been computed explicitly. It has been known that the smallest non-vanishing entropy among $3$-braids
equals $\log \frac{3+\sqrt{5}}{2}$ and is attained on the braid $\sigma_1^{-1} \sigma_2$. Further, there is an algorithm
which detects in principle whether a braid (respectively, an irreducible mapping
class) is pseudo-Anosov  and
in this
case it gives a computer assisted computation of the entropy (\cite{BH}). This approach uses so-called train
tracks. Further, fluid mechanics related to stirring devises
uses the entropy of the arising braids as a measure of complexity.

\noindent {\bf The conformal module of conjugacy classes of braids.}
On the other hand, the conformal module (or its inverse, the extremal length) is a conformal invariant of conjugacy classes of braids, and is even older than the entropy. The conformal module of conjugacy classes of braids
appeared first (without name) in the paper \cite{GL} in connection with the interest of the authors in Hilbert's Thirteen's Problem. This invariant of conjugacy classes of braids
is undeservedly almost forgotten, although it appeared again in Gromov's seminal paper \cite{G}.
Gromov observed
that the conformal module of conjugacy classes of braids defines restrictions for the validity of his Oka Principle
concerning homotopies of continuous or smooth objects involving braids to the respective holomorphic objects.

We recall the definition of the conformal module of conjugacy classes of braids. (See Definition \ref{def4.2} for the definition of the conformal module of the conjugacy classes of elements of the fundamental group of a complex manifold.)
\index{conformal module ! of a conjugacy class of braids}

We say that a continuous mapping $f$ of a round
annulus $A= \{z \in \mathbb{C}: \, r<|z|<R\},\;$ $0\leq r < R \leq \infty,\;$
into $ C_n(\mathbb{C})\diagup \mathcal{S}_n  $ represents an element $\hat b\in \widehat{ \mathcal{B}}_n$ if for some (and hence for
any) circle $\,\{|z|=\rho \} \subset A\,$ the loop $\,f:\{|z|=\rho
\} \rightarrow  C_n(\mathbb{C})\diagup \mathcal{S}_n\,$ represents $\,\hat b$.

Let $\hat b$ be a
conjugacy class of $n$-braids, $n \geq 2$. The conformal module
$M(\hat b)$ of $\hat b$ is defined as $ M(\hat b)=
\sup_{\mathcal{A}(\hat b)}\, m(A),$ where $\mathcal{A}(\hat b)$
denotes the set of all annuli which admit a holomorphic mapping into
$ C_n(\mathbb{C})\diagup \mathcal{S}_n$ which represents $\hat b$.
The extremal length $\Lambda(\hat b)$  of $\hat b$ is defined as
$\Lambda(\hat b)=
\inf_{\mathcal{A}(\hat b)}\, \lambda(A)$.
\index{$M(\hat{b})$}

\smallskip

\noindent {\bf The relation between entropy and conformal module.}
The interesting point is that the two invariants of conjugacy classes of braids,
the entropy, which is a dynamical invariant, and the conformal module, which is a conformal invariant related to complex and algebraic geometry, carry the same information on the braid class. More precisely, the following theorem holds.

\bigskip

\noindent {\bf Main Theorem.}  {\it For each conjugacy class of braids
$\hat b \in \hat{\mathcal B}_n$, $n \geq 2$, the following equality
holds
$$
{\mathcal M} (\hat b) = \frac\pi2 \, \frac1{h(\hat b)} \, .
$$
Equivalently,
$$
\Lambda (\hat b)= \frac{2}{\pi} h(\hat b)\,.
$$
}
\bigskip

The Main Theorem allows to apply methods and results known for the more intensively studied entropy to problems whose solutions are based on the concept of the conformal module. For instance, the Main Theorem together with the lower bounds for the entropy of irreducible $n$-braids allow to give a new conceptional proof of a slightly improved version of the Theorem of \cite{GL}, which was the first theorem that used the concept of conformal module (see Chapter \ref{chapter8}).
The Main Theorem together with Example 5 of Section \ref{sec:4.1b} give another proof of the fact that the smallest non-vanishing entropy among $3$-braids equals $\log\frac{3+\sqrt{5}}{2}$.
Methods of quasi-conformal mappings related to the concept of extremal length and conformal module are applied to give upper and lower bounds, differing by a universal multiplicative constant, for the entropy of $3$-braids (see Chapter \ref{chapter3-braids}).

In the present chapter we will give a proof of the Main Theorem for the case of irreducible braids. In Chapter \ref{chapter6} we will give its proof for arbitrary pure braids, and in Chapter \ref{chapter7a} the proof of the general case
will be completed.

In Chapters   \ref{chapter8}, \ref{chapterGrom}, \ref{chapterEl}, \ref{chapter3-braids}, \ref{Ch9}, and \ref{chapterfin}   we will give applications of the concept of the conformal module.
Chapters \ref{chapterGrom}, \ref{chapterEl}, and \ref{chapterfin}  are devoted to Gromov's Oka Principle. In these Chapters we apply the concept of the conformal module to describe restrictions for the existence of homotopies of continuous objects involving braids to the respective holomorphic objects.

\bigskip

\section [Upper bound for the
conformal module. Irreducible case. ]{Main Theorem. The upper bound for the
conformal module. The irreducible case.}
\label{sec:4.1}

\bigskip

\noindent Let $\hat b \in \hat{\mathcal B}_n$ be a conjugacy class of (maybe,  reducible) braids
and let ${\sf{f}} : A \to C_n ({\mathbb C}) \diagup {\mathcal S}_n$
be a holomorphic mapping of an annulus $A$ into the symmetrized
configuration space which represents $\hat b$. Our goal is to give an
upper bound for the conformal module $m(A)$.

It will be convenient to identify the annulus with the quotient of the upper half-plane by an automorphism of the upper half-pane as follows.
For a number $\rho >1$ we denote by $\Lambda_{\rho}$ the linear map
$z \to \rho z$ on the upper half-plane ${\mathbb C}_+ = \{ z \in
{\mathbb C} : {\rm Im} \, z > 0 \}$. The quotient ${\mathbb C}_+
\diagup \Lambda_{\rho}$ is conformally equivalent to an annulus of
conformal module $\frac{\pi}{\log \rho}$. Indeed, the half-open curvilinear
rectangle $\{ re^{i\theta} : 1 \leq r < \rho \, , \ 0 < \theta < \pi
\}$ is a fundamental polygon for the covering $\Lambda_{\rho}:{\mathbb C}_+\to
{\mathbb C}_+ \diagup \Lambda_{\rho}$. The logarithm maps it to $\{ x+iy : 0
\leq x < \log \rho \, , \ 0 < y < \pi\}$. Identifying points on the
vertical sides with equal $y$-coordinate we obtain an annulus of
conformal module $\frac{\pi}{\log\rho}$.
 \index{$\Lambda_{\rho}$}

\bigskip

\noindent {\bf Royden's Theorem and translation length.}
The key ingredient for obtaining the upper bound for the conformal
module is Royden's Theorem \ref{thm2.4} on equality of the Kobayashi and the Teichm\"uller metric on the Teichm\"uller space ${\mathcal T}
(0,n+1)$. Let $d_{\rm hyp}$ be the hyperbolic metric $\frac{\vert dz
\vert}{2y}$ on ${\mathbb C}_+$.
\index{Theorem ! Royden}
By Royden's Theorem \ref{thm2.4} any holomorphic
mapping ${\mathcal F}: \mathbb{C}_+ \to {\mathcal T}(0,n+1)$ is a contraction from $({\mathbb C}_+ ,
d_{\rm hyp})$ to $({\mathcal T} (0,n+1), d_{\mathcal T})$. In
particular, for any positive number $\rho$
\begin{equation}\label{eq4.6}
d_{\mathcal T} ({\mathcal F} (i) , {\mathcal F} (\rho i)) \leq
d_{\rm hyp} (i,\rho i) = \frac12 \log \rho \, .
\end{equation}
Let $\varphi$ be a self-homeomorphism of $\mathbb{P}^1$ with set of distinguished points $E_n^0\cup \{\infty\}$, and $\varphi^*$ its modular transformation on ${\mathcal T}(0,n+1)$. Suppose $\mathcal{F}$ has the following invariance property
\begin{equation}\label{eq4.6'}
\mathcal{F}(\rho z)= \varphi^*(\mathcal{F}(z)),\, z \in \mathbb{C}_+\,.
\end{equation}
The invariance property allows to associate to $\mathcal{F}$ a holomorphic map of the annulus ${\mathbb C}_+ \diagup {z \sim \rho z}$  of conformal module $\frac{\pi}{\log\rho}$ to the quotient ${\mathcal T}(0,n+1) \diagup ({\tau \sim \varphi^*(\tau)})$.
By Royden's Theorem
\begin{equation}\label{eq4.6''}
\frac{1}{2} \log \rho \geq L(\varphi^*)\,.
\end{equation}
Indeed, for the term on the left hand side of \eqref{eq4.6} the inequality
\begin{eqnarray}
\label{eq4.7} d_{\mathcal T} ({\mathcal F} (i) , {\mathcal
F} (\rho \, i)) &
= &d_{\mathcal T} ({\mathcal F} (i) , \varphi^* ( {\mathcal F} (i))) \nonumber \\
& \geq &\underset{\tau \in {\mathcal T} (0,n+1)}{\rm inf} \,
d_{\mathcal T} (\tau , \varphi^* (\tau)) = L(
\varphi^*)
\end{eqnarray}
holds.

\bigskip

\noindent {\bf The mapping to $\mathcal{T}(0,n+1)$ associated to a mapping to $C_n(\mathbb{C})\diagup \mathcal{S}_n$.}
The relation between the symmetrized configuration space
$C_n(\mathbb{C})\diagup \mathcal{S}_n$
and the Teichm\"{u}ller space of the $(n+1)$-punctured Riemann sphere (see Chapter \ref{chapter2})
will allow us to apply Teichm\"uller theory.

Take any holomorphic mapping ${\sf f}: \mathbb{C}\diagup \Lambda_{\rho}\to C_n(\mathbb{C})\diagup \mathcal{S}_n$. The set $\mathbb{C}\diagup \Lambda_{\rho}$ is identified with an annulus of conformal module $\frac{\pi}{\log\rho}$.
Lift $\sf{f}$ to a
$\Lambda_{\rho}$-equivariant mapping from ${\mathbb C}_+$ to $C_n
({\mathbb C}) \diagup {\mathcal S}_n$ which we denote by $f$.

Take a lift $\tilde f : {\mathbb C}_+ \to C_n
({\mathbb C})$ of $f$,
i.e. a mapping  $\tilde f$ for which ${\mathcal P}_{\rm
sym} \tilde f = f$.
The mapping ${\mathcal P}_{\mathcal A} \tilde f$ is a
holomorphic mapping from ${\mathbb C}_+$ to $C_n ({\mathbb
C})\diagup {\mathcal A}$. For the holomorphic isomorphism ${\rm Is}_{\sf n}:C_n(\mathbb{C})\diagup \mathcal{A}\to \{0\}\times \{\frac{1}{n}\}\times C_{n-2}(\mathbb{C}\setminus\{-1,1\})$ we lift the mapping $ {\mathcal P}_{\mathcal
{A},{\rm n}}  \tilde{f}\stackrel{def}={\rm Is}_{\rm n}\circ {\mathcal P}_{\mathcal
{A}}  \tilde{f}$
with respect to the holomorphic
projection
${\mathcal P}_{\mathcal T} : {\mathcal T} (0,n+1) \to C_n ({\mathbb
C})\diagup  \{0\}\times \{\frac{1}{n}\}\times C_{n-2}(\mathbb{C}\setminus\{-1,1\})    $ and obtain
a holomorphic map ${\mathcal F} : {\mathbb C}_+ \to
{\mathcal T} (0,n+1)$, such that ${\mathcal P}_{\mathcal T}
{\mathcal F} = {\mathcal P}_{\mathcal{ A},{\rm n}} \tilde{f}$. We have the
commutative diagram Figure \ref{fig4.1}.
(For the definition of ${\mathcal P}_{\rm
sym}$, ${\mathcal P}_{\mathcal{ A},{\rm n}}$ and ${\rm Is}_{\rm n}$ see Chapter \ref{chapter2}.)
\index{${\rm Is}_{\rm n}$}
\begin{figure}[h]
\begin{center}
\includegraphics[width=9.5cm]{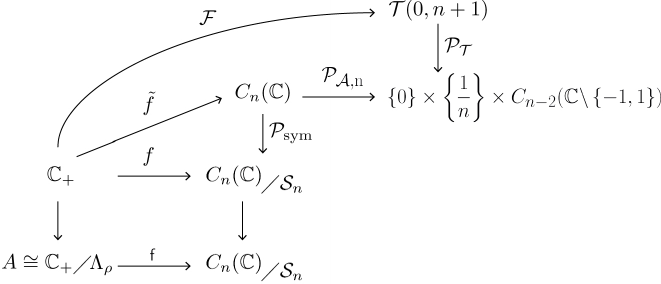}
\end{center}
\caption{The holomorphic mapping to Teichm\"uller space associated to a braid.}\label{fig4.1}
\end{figure}

Denote by $E_n$ the point
$E_n = f(i) = f(i \rho)$ and by $(z_1 , \ldots , z_n) \in C_n ({\mathbb C})$ the point $\tilde{f}(i)$.
Note that ${\mathcal P}_{\rm sym} ((z_1 , \ldots , z_n)) = E_n$.
For notational convenience we put $E_n'=E_n\cup\{\infty\}$.
Choose a
smooth self-diffeomorphism $\psi_i$ of ${\mathbb P}^1$ that is the
identity outside a large disc containing $E_n^0=\{0,\frac{1}{n},\ldots,\frac{n-1}{n}\}$ and $E_n$, and maps %
$\infty$ to $\infty$, and
$\frac{j-1}n$ to $z_j$ for $j = 1,\ldots ,n$, and has the property
${\mathcal F} (i) = [\psi_i]$.

\bigskip

Recall that each braid $b$ with base point $E_n^0$ corresponds to a mapping class $\mathfrak{m}_b$ of self-homeomorphisms of the unit disc which fix the set $E_n^0$ setwise and the boundary circle pointwise. We denoted by $\varphi^*_{b,\infty}$ the modular transformation on $\mathcal{T}(0,n+1)$ of an element of the mapping class $\mathfrak{m}_{b,\infty} \in \mathcal{T}(0,n+1)$. Recall that a representative of the mapping class $\mathfrak{m}_{b,\infty}$ may be obtained by extending an element of $\mathfrak{m}_b$ to the Riemann sphere putting it equal to the identity outside the disc.
For a self-homeomorphism $\psi$ of $\mathbb{P}^1$ that fixes $\infty$ and an element $\{\zeta_1,\ldots,\zeta_n\}\in C_n(\mathbb{C})\diagup \mathcal{S}_n$ we put as before $\psi(\{\zeta_1,\ldots,\zeta_n\})=\{\psi(\zeta_1),\ldots,\psi(\zeta_n)\}$.

The following lemma holds.
\begin{lemm}\label{lemm4.1} Let $\tilde{f}:\mathbb{C}_+ \to C_n(\mathbb{C})$ be a smooth mapping such that $f=\mathcal{P}_{\rm{sym}}(\tilde{f} )$ is $\Lambda_{\rho}$-invariant.
Let $\mathcal{F}:\mathbb{C}_+\to \mathcal{T}(0,n+1)$  be a smooth  mapping such that the diagram Figure  $\ref{fig4.1}$ commutes.
Denote by $b$ the braid
with base point $E^0_n$ represented by the geometric braid
\begin{equation}\label{eq4.19}
\psi_i^{-1} \circ f \mid [i,i \rho],\, t\in [1,\rho]
\end{equation}
where $\psi_i \in \mathcal{F}(i) $ is the self-diffeomorphism of $\mathbb{P}^1$ that was chosen above.
Then for all $z \in \mathbb{C}_+$
\begin{equation}\label{eq4.1}
\mathcal{F} (\rho \,z) = \varphi_{b,\infty}^* (\mathcal {F} (z)).
\end{equation}
\end{lemm}
\noindent{\bf Proof.}
Let $z$ be an arbitrary point in $ \mathbb{C}_+$,
and let $\alpha_z$ be a smooth curve in $\mathbb{C}_+$ that joins the points $i$ and $z$. By the proof of Proposition \ref{prop2.2} (applied to $\psi_i ^{-1}\circ f$)
there exists a smooth family of self-diffeomorphisms $\psi_{\zeta}$ of $\mathbb{P}^1$,    $ \zeta \in \alpha_z,$  that are equal to the identity outside a large disc,  such that $\psi_{\zeta}(E_n^0)=f(\zeta)$, 
and for $\zeta=i$ the diffeomorphism of the family coincides with the diffeomorphism $\psi_i$ in the statement of the lemma.

As a consequence $[\psi_{\zeta}]=\mathcal{F}(\zeta)$ since both mappings, $\zeta\to [\psi_{\zeta}]$ and
$\zeta \to \mathcal{F}(\zeta)$ lift $  \mathcal{P}_{\mathcal{A},{\rm n}}( \tilde{f}(\zeta))$ by the commutative diagram of Lemma \ref{lemm2.5}.

The family
\begin{equation}\label{eq4.**}
\zeta \to \Big\{\big(t,\, \psi_{\zeta}^{-1}(f(\zeta t))\big),\, t \in[1,\rho]\Big\}, \zeta \in \alpha_z\,,
\end{equation}
is an isotopy of geometric braids with base point $E_n^0$ that joins the geometric braid \eqref{eq4.19} with the geometric braid
\begin{equation}\label{eq4.20}
\Big\{\big(t,\, \psi_{z}^{-1}(f(z t))\big),\, t \in[1,\rho]\Big\}\,.
\end{equation}

Fix $z \in \mathbb{C}_+$ and take a parameterizing isotopy for the geometric braid \eqref{eq4.20}. More precisely, for a large positive number $R$ we take a smooth family of homeomorphisms
$\varphi^z_t \in {\rm Hom}^+ ( \mathbb{P}^1;\,\mathbb{P}^1 \setminus R\mathbb {D}),\,t
\in [1,\rho], $ such that
\begin{equation}\label{eq4.2}
{\rm ev}_{E_n^0} \, \varphi^z_t = \psi_z^{-1} (f(z\,t)) \, , \quad t \in
[1,\rho] \, ,
\end{equation}
and $\varphi^{z}_1$ is the identity.
Then the mapping $\varphi^i_{\rho}$
represents the mapping class $\mathfrak{m}_{b,\infty}\in{\mathfrak M} (\mathbb{P}^1;\,\{\infty\},\, E_n^0)$. Write
$\varphi^z_{\rho}=\varphi^z_{b,\infty}$. Since the $\varphi^z_{b,\infty}$ are isotopic with base point $E_n^0$, their modular transformations do not depend on $z$.
Let
$\varphi_{b,\infty}^*$ be the modular transformation induced on
${\mathcal T} (0,n+1)$ by the $\varphi_{b,\infty}^z$.

Equation \eqref{eq4.2} can also be written as ${\rm ev}_{E_n^0}
\,( \psi_z \circ \varphi^z_t) = f(z\,t)$, $t \in [1,\rho]$.
Hence, the mapping $t\to
e_n( \psi_z \circ \varphi^z_t) $ is a lift of the mapping $t\to f(zt)$, $t\in[0,1]$. Clearly also $t\to \tilde{f}(zt), t\in[0,1]$, lifts
the mapping $t\to f(zt)$, $t\in[0,1]$.
Moreover, both mappings, $e_n(\psi_z \circ \varphi_t^z)$ and $\tilde{f} (z\,t)$, take the value $\tilde{f}(i)=(z_1,\ldots,z_n)$ for $t=1$.
For $e_n (\psi_z \circ
\varphi^z_t)$ this follows from the definition of $\psi_i$ and
the fact that $\varphi^z_{1} = {\rm id}$. Hence, $e_n
(\psi_z \circ \varphi_t^z) = \tilde{f} (z\,t)$, $t \in [1,\rho]$.
Using also Lemma \ref{lemm2.5} we obtain the commutative
diagram Fig. \ref{fig4.3}.

\begin{figure}[h]
\begin{center}
\includegraphics[width=9.5cm]{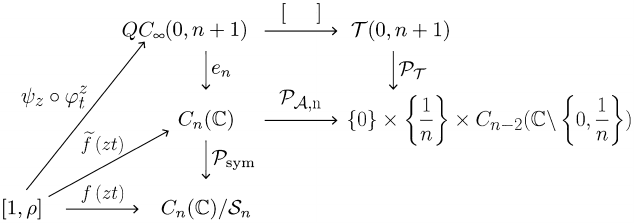}
\end{center}
\caption{}\label{fig4.3}
\end{figure}

The diagrams Figure \ref{fig4.1} and Figure \ref{fig4.3} show that
both mappings, $t\to\mathcal{F}(z\,t)$ and $t \to [\psi_z \circ \varphi_t^z]$,
lift the mapping $t\to {\mathcal P}_{\mathcal A,{\sf n}} ( \tilde f(z\,t))$ to a mapping from
$[1,\rho]$ to the Teichm\"uller space ${\mathcal T} (0,n+1)$.
Moreover, since
$\mathcal {F} (z) = [\psi_z]$ and $\varphi^z_{1} =
{\rm id}$, we may write ${\mathcal F} (z) = [\psi_z \circ
\varphi_1^z]$, which is the value of $[\psi_z \circ
\varphi^z_t]$ for $t=1$. Therefore, the two lifts
coincide:
\begin{equation}\label{eq4.4}
{\mathcal F} (z\,t) = [\psi_z \circ \varphi_t^z] \, , \quad t \in
[1,\rho] \, .
\end{equation}
In particular,
\begin{equation}\label{eq4.5}
{\mathcal F}(\rho\,z) = [\psi_z \circ \varphi^z_{\rho}] = [\psi_z \circ
\varphi_{b,\infty}^z] = (\varphi_{b,\infty})^* ([\psi_z]) =
\varphi_{b,\infty}^* ({\mathcal F} (z)) \, .
\end{equation}
We obtained \eqref{eq4.1}
for the arbitrarily chosen point $z$.
The lemma is proved. \hfill $\Box$

\medskip

\noindent {\bf Translation length and the upper bound for the conformal module. The general case.}
The following proposition gives an upper bound for the conformal module of a conjugacy class of braids in terms of the translation length of the associated modular transformation that holds for all (maybe, reducible) $n$-braids.
\begin{prop}\label{prop4.2}
Let $b$ be an $n$-braid with base point $E_n^0$ and ${\mathfrak m}_b$ its
mapping class (considered as mapping class on a large $n$-punctured disc).
For the associated mapping class  ${\mathfrak m}_{b,\infty} =
\mathcal{H}_{\infty}({\mathfrak m}_b) \, \in \mathfrak{M}(\,
\mathbb{P}^1; \, \infty,\, E_n^0)\,$ on the Riemann sphere
we consider a representative
$\varphi_{b,\infty}$ of ${\mathfrak m}_{b,\infty}\,$, its  modular
transformation $\varphi_{b,\infty}^*\,$ on ${\mathcal T} (0,n+1)\,$,
and the translation length $L (\varphi_{b,\infty}^*)$.
Then
\begin{equation}\label{eq4.8}
L (\varphi_{b,\infty}^*) \leq
 \frac{\pi}{2} \,
\frac{1}{\mathcal{M}(\hat b)} \, .
\end{equation}
\end{prop}
\index{$\mathfrak{m}_{b,\infty}$}
\medskip

\noindent{\bf Proof.} As before we take an arbitrary holomorphic mapping $\sf{f}$ that takes the annulus $A=\mathbb{C}_+\diagup \Lambda_{\rho}$ of conformal module $m(A)=\frac{\pi}{\log \rho}$ to the symmetrized configuration space, and represents the conjugacy class $\hat b$ of the braid $b$.
Then by Lemma \ref{lemm4.1} and Royden's Theorem the inequality \eqref{eq4.6''} holds for the modular transformation $\varphi^*_{b,\infty}$ and the mapping $\mathcal{F}$ associated to $\sf f$ by the diagram Figure \ref{fig4.1}. Hence, in terms of the conformal module of $A$
$$
\frac{1}{2}\frac{\pi}{m(A)}\geq L(\varphi_{b,\infty}^*)\,.
$$
Taking the supremum over the conformal modules of all annuli admitting a holomorphic mapping representing $\hat b$, we obtain \eqref{eq4.8}. \hfill $\Box$

\bigskip

\noindent {\bf Entropy and the upper bound for the conformal module. The irreducible case.} Proposition \ref{prop4.2} implies the upper bound for the conformal module of \textit{irreducible} braids in terms of the entropy. More precisely, the following proposition holds.
\begin{prop}\label{prop4.3} For each irreducible
conjugacy class of braids $\hat b \in \hat{\mathcal B}_n$ the
inequality
$$
h(\hat b) \leq \frac\pi2 \, \frac1{{\mathcal M} (\hat b)}
$$
holds.
\end{prop}

\noindent {\bf Proof.}
If
a braid $b$ with base point $E_n^0$ is irreducible, then,
$\varphi_{b,\infty}$ is irreducible and therefore
$\varphi_{b,\infty}^*$ is either elliptic or hyperbolic. By
Corollary \ref{corr2.1} there is an absolutely extremal
self-homeomorphism $\tilde\varphi_{b,\infty}$ of a Riemann surface
${\mathbb C} \backslash \tilde{E}_n$ which is obtained from
$\varphi_{b,\infty}$ by isotopy and conjugation. More precisely,
there is a homeomorphism $w : {\mathbb C} \backslash E_n^0 \to
{\mathbb C} \backslash \tilde{E}_n$ and a self-homeomorphism
$\hat\varphi_{b,\infty}$ of ${\mathbb C} \backslash E_n^0$ which is
isotopic to $\varphi_{b,\infty}$ on ${\mathbb C} \backslash E_n^0$
so that $\tilde\varphi_{b,\infty} = w \circ \hat\varphi_{b,\infty}
\circ w^{-1} : {\mathbb C} \backslash \tilde{E}_n \to {\mathbb C}
\backslash \tilde{E}_n$ is absolutely extremal. $\tilde\varphi_{b,\infty}$
is pseudo-Anosov if $\varphi_{b,\infty}^*$ is hyperbolic and
conformal if $\varphi_{b,\infty}^*$ is elliptic. For
$\tilde\varphi_{b,\infty}$ we have by Theorem \ref{thm3.2} and
Corollary \ref{corr2.1}
\begin{equation}\label{eq4.7'}
\frac12 \log K (\tilde\varphi_{b,\infty})\, =\,
 L (\varphi_{b,\infty}^*)\,=\,
h(\tilde\varphi_{b,\infty}) \, =\, h(\widehat { {\mathfrak
m}_{b,\infty}})  \, .
\end{equation}
Theorem \ref{thm3.5} provides an isotopy of $\tilde\varphi_{b,\infty}$ through
self-homeomorphisms of ${\mathbb C} \backslash \tilde{E}_n$ to a
homeomorphism $\tilde\varphi_b$ which is the identity outside the
disc $R \, \overline{\mathbb D}$ and has the same entropy as
$\tilde\varphi_{b,\infty}$:
\begin{equation}\label{eq4.8'}
h(\tilde\varphi_b) = h(\tilde\varphi_{b,\infty}) \, .
\end{equation}
Then $w^{-1} \circ \tilde\varphi_b \circ w$ is isotopic to
$\varphi_{b,\infty}$ through self-homeomorphisms of ${\mathbb C}
\backslash E_n^0$ and is the identity outside the disc $R \,
\overline{\mathbb D}$. By Theorem \ref{thm2.1} the mapping class
$\mathfrak{m}_b$ associated to $b$, and the mapping class of
$w^{-1} \circ \tilde\varphi_b \circ w \mid R \, \overline{\mathbb
D}\;$ in $\; {\mathfrak M} (R \, \overline{\mathbb D} ;
\partial { ( R \,
\overline{\mathbb D})}
, E_n^0)$ differ by a power of a Dehn twist about a circle in $R \,
{\mathbb D}$ of large radius. By
Corollary \ref{cor3.2}
the entropies of the two mapping classes are equal. We obtain
\begin{align}\label{eq4.9}
h(\hat b) = h(b)  = {\rm inf} \{ h(\varphi) : \varphi \in {\mathfrak m}_b \} & = {\rm inf} \{ h(\varphi) : \varphi \in {\mathfrak m}_{w^{-1}
\circ \tilde\varphi_b \circ w \mid R {\mathbb D}}\}    \nonumber \\
& \leq h (w^{-1}
\circ \tilde\varphi_b \circ w \mid R \, {\mathbb D}) \, .
\end{align}
Hence we obtain from \eqref{eq4.7'}, \eqref{eq4.8'}, and \eqref{eq4.9}
$$
h(\hat b) \leq h(w^{-1} \circ \tilde\varphi_b \circ w) =
h(\tilde\varphi_b) = h(\tilde\varphi_{b,\infty}) = L
(\varphi_{b,\infty}^*) \, .
$$
By \eqref{eq4.8} we have $h(\hat b) \leq \frac\pi2 \,
\frac1{\mathcal M(\hat b)}$. The proposition is proved. \hfill $\Box$

\section [Lower bound for the
conformal module. Irreducible case.] {The lower bound for the
conformal module. The irreducible case.}
\label{sec:4.2}

\noindent We will prove here the following proposition for the irreducible case.

\begin{prop}\label{prop5.1} Let $\hat b
\in \hat{\mathcal B}_n$ be an irreducible conjugacy class of braids.
Then for any annulus $A$ of conformal module $m(A)=\frac\pi 2 \frac1{h(\hat b)}$ the class $\hat b$ can be represented by a holomorphic map $f
: A \to C_n ({\mathbb C}) \diagup {\mathcal S}_n$.
\end{prop}

In the irreducible case Proposition \ref{prop5.1} is stronger than the statement of the Main Theorem. Indeed, the Proposition \ref{prop5.1} asserts furthermore that in the irreducible case the supremum in Definition 1 is attained.
Moreover, Proposition \ref{prop5.1} says, that in case the supremum is infinite, it is attained on the
punctured complex plane, not merely on the punctured disc.

Take an irreducible conjugacy class of $n$-braids $\hat b \in \hat{\mathcal
B}_n$.
Let $b \in {\mathcal B}_n$ be a braid with base point $E_n^0$
representing $\hat b$.
Let $\varphi_b$ be a self-homeomorphism of
$\overline{\mathbb D} \setminus E_n^0$ which represents the mapping
class ${\mathfrak m}_b = \Theta_n(b) \in {\mathfrak M}
(\overline{\mathbb D} ;
\partial \, {\mathbb D} , E_n^0)$ of $b$. Let $\varphi_{b,\infty}
\in {\rm Hom}^+ ({\mathbb C} ; \emptyset , E_n^0)$ be the extension
of $\varphi_b$ to the whole plane which is the identity outside the
unit disc.

\smallskip

Since the class $\hat b$ is irreducible, the mapping
$\varphi_{b,\infty}$ is irreducible. By Theorem \ref{thm2.11} the
induced modular transformation $\varphi_{b,\infty}^*$ on the
Teichm\"uller space ${\mathcal T} (0,n+1)$ is either elliptic or
hyperbolic.

\medskip

\noindent {\bf Proof of Proposition \ref{prop5.1}. The elliptic case.}
By Theorem \ref{thm2.10} the modular transformation $\varphi_{b,\infty}^*$ is elliptic iff there is a periodic conformal map
$\tilde{\varphi}_{b,\infty} \in {\rm Hom}^+ ({\mathbb C} ; \emptyset , E_n)$
which is obtained from $\varphi_{b,\infty}$ by isotopy and
conjugation. Here $E_n \subset {\mathbb C}$ is a set consisting of
$n$ points. The mapping $\tilde{\varphi}_{b,\infty}$ is a M\"obius
transformation that fixes $\infty$ and, being periodic, has a fixed
point on ${\mathbb P}^1$ different from $\infty$. Conjugate this
fixed point to zero. We obtain that $\tilde{\varphi}_{b,\infty}$ is conjugate
by a conformal mapping to multiplication by a root of unity $\omega
= e^{2\pi i \frac{{\sf p}}{{\sf q}}}$ for integer numbers ${\sf p}$
and ${\sf q}$, ${\sf q} \ne 0$. We may assume that the
mapping $\tilde{\varphi}_{b,\infty}$ itself has the form $\tilde{\varphi}_{b,\infty} (z) =
e^{2\pi i \frac{{\sf p}}{{\sf q}}} z$. Since $\tilde{\varphi}_{b,\infty} (E_n) = E_n$, the set $E_n$ is either equal to $\left\{ r , r \, e^{\frac{2\pi
i}{n}} , \ldots , r \, e^{\frac{2\pi i (n-1)}{n}} \right\}$ for some
$r > 0$, and in this case $\omega = e^{\frac{2\pi i {\sf p}}{n}}$; or it is equal
to $E_n = \left\{ 0,re^{\frac{2\pi i}{n-1}} , \ldots , r \,
e^{\frac{2\pi i (n-2)}{n-1}} \right\}$ for some $r > 0$, and in this
case $\omega = e^{\frac{2\pi i {\sf p}}{n-1}}$.

\smallskip

Consider the universal covering
$$
{\mathbb C} \ni \zeta = \xi + i\eta \to e^{\xi + i\eta} \in {\mathbb C}^*
$$
of the annulus ${\mathbb C}^* = {\mathbb C} \backslash \{ 0 \} = \{
z \in {\mathbb C} : 0 < \vert z \vert < \infty \}$.

\smallskip

Denote by $\tilde{\mathcal F}$ the following holomorphic mapping
from ${\mathbb C}$ into the space ${\mathcal A}$ of complex affine
mappings,
\begin{equation}\label{eq5.1}
\tilde{\mathcal F} (\zeta) = {\mathfrak a} (\zeta) \in {\mathcal A}
\, ,  \mbox{where} \quad {\mathfrak a} (\zeta) (z) = e^{\frac{{\sf p}}{{\sf q}} \cdot \zeta} \cdot z \, .
\end{equation}
Notice that
\begin{equation}\label{eq5.2}
{\mathfrak a} (0) (z) =  z \quad  \mbox{and} \quad {\mathfrak a} (2\pi i) (z) = e^{2\pi i \frac{{\sf p}}{{\sf q}}} z = \tilde{\varphi}_{b,\infty} (z) \, .
\end{equation}

\smallskip

The evaluation map
\begin{equation}\label{eq5.3}
{\mathbb C} \ni \zeta \to {\rm ev}_{E_n} \, {\mathfrak a}(\zeta) \in
C_n ({\mathbb C}) \diagup {\mathcal S}_n
\end{equation}
is holomorphic and is periodic with period $2\pi i$. Hence, the
evaluation map induces a holomorphic map $f$
from ${\mathbb C}^*$ to $C_n ({\mathbb C})\diagup {\mathcal S}_n$.
$f$ represents a conjugacy class $\hat b' \in
\hat{\mathcal B}_n$. By Theorem \ref{thm2.1} the mapping classes
corresponding to the braids $b' \in \hat b'$, and $b \in \hat b$
differ by a power of a Dehn twist. Replacing ${\sf p}$ by ${\sf p} +
2\pi k {\sf q}$ for a suitable integer number $k$ we may achieve
that $f : {\mathbb C}^* \to C_n ({\mathbb C}) \diagup
{\mathcal S}_n$ represents the conjugacy class $\hat b \in
\hat{\mathcal B}_n$. The proposition is proved for the elliptic case. \hfill $\Box$

Note that in the elliptic case $r=\infty$, and we may take $A = {\mathbb C}^*$.

\medskip

\noindent {\bf
The hyperbolic case. Teichm\"uller discs.} The plan of the proof in the hyperbolic case is the following. We will associate to the irreducible braid $b$ a Teichm\"uller disc, i.e. a holomorphic map $\mathcal{F}: \mathbb{C}_+\to \mathcal{T}(0,n+1)$ (see Figure \ref{fig4.1}). We need to find a mapping $\tilde f: \mathbb{C}_+ \to C_n(\mathbb{C})$ for which $\mathcal{P}_ {\mathcal{A},{\sf n}} \tilde{ f} = \mathcal{P}_{\mathcal{T}} \mathcal{F}$ (which means to find a holomorphic section), such that  $\mathcal{P}_{\rm sym}\tilde{f}$ has the equivariance property for which the Diagram
\ref{fig4.1} is commutative.

Associate to the irreducible braid $b \in \hat b$ with base point $\,E_n^0 = \left\{
0, \frac1n , \ldots , \frac{n-1}n \right\}\,$ a
representative $\,\varphi_b\,$ of the mapping class $\;{\mathfrak m}_b \in {\mathfrak
M} (\overline{\mathbb D} ; \partial  {\mathbb D} , E_n^0)\,.\;$ Let
$\varphi_{b,\infty} \in {\rm Hom}^+ ({\mathbb C} ; \emptyset ,
E_n^0)$ be the extension of $\varphi_b$ to ${\mathbb C}$ which is the
identity outside the unit disc ${\mathbb D}$. Consider the modular transformation $\varphi_{b,\infty}^*$ on $\mathcal{T}(0,n+1)$ that is associated to $b$, and let $\tilde{\varphi}_{b,\infty}$ be the absolutely extremal map that is obtained from $\varphi_{b,\infty}$ by isotopy and conjugation. Put $X_0=\mathbb{C}\setminus E_n^0= \mathbb{P}^1\setminus ({E_n^0})'$.
Denote by $w_0: X_0 \to w_0(X_0)=X$ the $\tilde{\varphi}_{b,\infty}$-minimal conformal structure on $X_0$. Let $X=w_0(X_0)$ and $E_n=w_0(E_n^0)$.
The mapping $ \tilde{\varphi}_{b,\infty}$
is a self-homeomorphism of $X$, equivalently, $\tilde{\varphi}_{b,\infty}$ defines a self-homeomorphism of $\mathbb{P}^1$ with set of distinguished points $E_n\cup \infty$, and $ \tilde{\varphi}_{b,\infty}$ is isotopic to $w_0 \circ \varphi_{b,\infty}\circ w_0^{-1}$.

By Corollary \ref{cor3.2}
$\;h (\tilde{\varphi}_{b,\infty})\; =\, h(\widehat {{\mathfrak
m}_{\varphi_{b,\infty}}}) = h(\widehat{b})$.
Again we identify self-homeomorphisms of punctured surfaces with
self-homeomorphisms of closed surfaces with distinguished points obtained by extension to the punctures.

\smallskip

The Beltrami differential of the self-homeomorphism
$\tilde{\varphi}_{b,\infty}$ of $X$ has the form $k \,
\frac{\bar\phi}{\vert \phi \vert}$ for a number $k\in (0,1)$ and a  meromorphic quadratic
differential $\phi$ on ${\mathbb P}^1$ which is holomorphic on
${\mathbb C} \backslash E_n$ and has at worst simple poles at the
points of $E_n \cup \{ \infty \}$.
The quasiconformal dilatation of $\tilde{\varphi}_{b,\infty}$
equals $K=\frac{1+k}{1-k}$. Recall that by Theorem \ref{thm3.2} we have
$\frac{1}{2} \log K = h(\tilde{\varphi}_{b,\infty}) $.

\smallskip
For $\mu_z \,= \,z \,
\frac{\bar\phi}{\vert\phi\vert}, \; z \in \mathbb{D},\;$
we let $W^{\mu_z}$ be the normalized solution on $X$ of the Beltrami
equation for $\mu_z$
(see Definition \ref{def2.1}).
Consider the Teichm\"uller disc $\mathbb{D}\ni z \to  [W^{\mu_z}]\in   {\mathcal D}_{ X, \phi}$ \index{$\mathcal{D}_{X,\phi}$}
in ${\mathcal T} (X)$ which is associated to
the quadratic differential $\phi$ on $X$ (see Section \ref{sec:2.3a}).
The modular transformation $\tilde{\varphi}_{b,\infty}^* $
maps the
disc ${\mathcal D}_{ X,\phi}$ onto itself. Indeed, $\tilde{\varphi}_{b,\infty}$ has Beltrami differential $k \,
\frac{\bar\phi}{\vert\phi\vert}$. 
For each $z$ in the open unit disc the mapping
$W^{\mu_z} \circ \, \tilde{\varphi}_{b,\infty}$ is a Teichm\"uller mapping
with quadratic differential ${\rm const} \cdot \phi$ (see Section \ref{sec:2.3a}). Hence $[W^{\mu_z} \circ \tilde{\varphi}_{b,\infty}] =
\tilde{\varphi}_{b,\infty}^* ([W^{\mu_z} ])$ has the form $\{ \mu_{z'}
\}$ for some $z' \in {\mathbb D}$. If $z$ is real then also $z'$ is
real. Each element $\{\mu_{z'}\},\, z' \in \mathbb{D}$, is in the image under $\tilde{\varphi}_{b,\infty}^* $ of the Teichm\"uller disc .
Since $W^0$ is the identity mapping, $\tilde{\varphi}_{b,\infty}^* $ maps
the point $\{\mu_0\}$ to $\{\mu_k\}$.

\smallskip

The conformal structure $w_0$ realizes a canonical
isomorphism between the Teichm\"uller space ${\mathcal T} (X)$ and
the canonical Teichm\"uller space ${\mathcal T} (X_0) = {\mathcal T}
(0,n+1)$. Indeed, associate to each conformal structure $w$ on $X$
the conformal structure $w \circ w_0$ on $X_0$. Its class $[w \circ
w_0]$ in ${\mathcal T} (0,n+1)$ depends only on the class $[w]$ in
${\mathcal T} (X)$. Put $w_0^* ([w]) = [w \circ w_0]$, $[w] \in
{\mathcal T} (X)$. The mapping $w_0^* : {\mathcal T} (X) \to
{\mathcal T} (0,n+1)$ gives the canonical isomorphism. $w_0^*$ is a
holomorphic mapping between Teichm\"uller spaces.
We denote the image  $w_0^* ({\mathcal D}_{ X,\phi})$ in ${\mathcal T}(0,n+1)$
by ${\mathcal D}_{\phi}^0$.
The mapping $w_0^*$ takes the Teichm\"uller disc in ${\mathcal T} (X)$ to a Teichm\"uller disc
\begin{equation}\label{eq5.4}
{\mathbb D} \ni z \overset{{\mathfrak
e}}{-\!\!\!-\!\!\!\longrightarrow} w_0^* (\{ \mu^z \}) \in
{\mathcal D}_{\phi}^0 \subset {\mathcal T} (0,n+1) \,
\end{equation}
in ${\mathcal T}(0,n+1)$. \index{$\mathcal{D}_{\phi}^0$}

It will be convenient to reparameterize the Teichm\"uller disc by a mapping from the upper half-plane to Teichm\"uller space. For this purpose we consider the composition
\begin{equation}\label{eq5.5}
\mathcal{F}: {\mathbb C}_+ \to {\mathcal D}_{\phi}^0 \, ,
\quad
\mathcal{F} = {\mathfrak e} \circ {\mathfrak
c}^{-1} \,
\end{equation}
with the conformal mapping ${\mathfrak c} (z) = i \, \frac{1+z}{1-z}$ from $\mathbb{D}$ onto ${\mathbb C}_+$.

The mapping $\varphi_{b,\infty}^*=(w_0^* )^{-1}\circ \tilde{\varphi}_{b,\infty}^*\circ w_0^* $ takes ${\mathcal D}^0_{\phi}$ onto itself.
Conjugate the restriction of the modular transformation $\varphi_{b,\infty}^* \mid {\mathcal
D}_{\phi}^0$ by $\mathcal{F}$. We obtain a holomorphic homeomorphism $\Lambda$ of the upper half-plane onto itself,
\begin{equation}\label{eq5.6}
\Lambda = \mathcal{F}^{-1} \circ \varphi_{b,\infty}^*
\circ \mathcal{F} : {\mathbb C}_+ \to {\mathbb C}_+\,.
\end{equation}
Since $\tilde{\varphi}_{b,\infty}^*$ maps each point $\{\mu_x\}$ with $x$ real to a point $\{\mu_{x'}\}$ with real $x'$,  and maps $\{\mu_0\}$ to $\{\mu_k\} $, the mapping
$\Lambda$ fixes the imaginary axis and maps $i$ to $i \, \frac{1+k}{1-k} = iK$. Hence
\begin{equation}\label{eq5.7}
\Lambda (\zeta) = \Lambda_K (\zeta) \stackrel{def}= K \zeta \, , \quad \zeta \in
{\mathbb C}_+.
\end{equation}
By \eqref{eq5.6} and \eqref{eq5.7}
\begin{equation}
\label{eq5.12} \mathcal{F}
(K\zeta) = \varphi_{b,\infty}^* (\mathcal{F} (\zeta)).
\end{equation}

The annulus $A =  {\mathbb C}_+ \diagup \Lambda_K$ has conformal
module $\frac\pi{2\log K^{\frac12}}$ and is conformally equivalent to the quotient ${\mathcal D}^0_{\phi} \diagup {\langle \varphi_{b,\infty}^*\rangle}$.

\bigskip

\noindent{\bf The mapping representing the braid class.}
Using a  mapping $\mathcal{F}: {\mathbb C}_+ \to {\mathcal D}^0_{\phi}$ that satisfies \eqref{eq5.12}, we will now represent the braid class $\hat b$ by a holomorphic mapping from an annulus to $C_n(\mathbb{C})\diagup \mathcal{S}_n$. Associate to $\mathcal{F}$  the  mapping $\mathcal{P}_{\mathcal{T}}\circ \mathcal{F}$ from $\mathbb{C}_+$ into $\{ 0 \} \times \{\frac1n\} \times C_{n-2} \, ({\mathbb C}\setminus \{1,\frac{1}{n}\}) ={\rm Is}_{\rm n}( C_n(\mathbb{C})\diagup \mathcal{A})$ (see Section \ref{sec:2.3a}). The diagram Figure \ref{fig4.1} and the diagram of Lemma  \ref {lemm2.5} suggest to find a suitable holomorphic section for the projection
\begin{align*}
\mathbb{C}_+ \times  C_n ({\mathbb C}) \ni (\zeta, z) \to (\zeta, z\diagup {\mathcal{A}})\in
\mathbb{C}_+\times
C_{n} ({\mathbb C})\diagup {\mathcal{A}}\,.
\end{align*}
We need a suitable section only  for $\mathbb{C}_+$ replaced by a neighbourhood
${\mathcal U}(\varepsilon)\stackrel{def} = \{ \zeta \in {\mathbb C}_+ : 1 -
\varepsilon < \vert \zeta \vert < (1+\varepsilon) K \}$ of the closure $\{ \zeta \in {\mathbb C}_+ : 1 \leq \vert
\zeta \vert \leq K\}$ of a fundamental domain in the universal covering $\mathbb{C}_+$ of the annulus $\mathbb{C}_+\diagup \Lambda$.
Each such holomorphic section has the following form
\begin{align*}
{\mathcal U}(\varepsilon)\times C_{n} ({\mathbb C})\diagup {\mathcal{A}} \ni (\zeta,Z)\to (\zeta,\mathfrak{A}(\zeta)({\rm Is}_{\sf n}(Z)))\in {\mathcal U}(\varepsilon)\times C_{n} ({\mathbb C})
\end{align*}
for a holomorphic map $\mathfrak{A}: {\mathcal U}(\varepsilon)\to \mathcal{A}$
that assigns to each point $\zeta\in {\mathcal U}(\varepsilon)$ a complex linear self-mapping $\mathfrak{A}(\zeta)\in \mathcal{A}$ of $\mathbb{C}$. Recall that for an element $\mathfrak{A}$ of $\mathcal{A}$ the point
\begin{align}
{\mathfrak A}\Big((0,\frac1n , z_3  , \ldots , z_n)\Big) \stackrel{def} =
\Big({\mathfrak A}(0), {\mathfrak A}(\frac1n) , {\mathfrak A} (z_3 ) , \ldots , {\mathfrak A} (z_n )\Big)
\end{align}
is the result of the diagonal action of the complex affine mapping $\mathfrak{A}$ on the element $(0,\frac1n , z_3  , \ldots , z_n)\in
\{ 0 \} \times \{\frac1n\} \times
C_{n-2} \, ({\mathbb C} \backslash \{0,\frac{1}{n}\})$.

For $\zeta \in {\mathbb C}_+$ we
write
\begin{equation}
\label{eq5.7a}
{\mathcal P}_{\mathcal T} ( \mathcal{F}
(\zeta)) \, = \, \big( 0,\frac1n , z_3 (\zeta) , \ldots , z_n
(\zeta)\big)\,.
\end{equation}

For a holomorphic mapping $\mathfrak{A}:{\mathcal U}(\varepsilon)\to \mathcal{A}$  we define the holomorphic mapping $f_{\mathfrak A}: {\mathcal U}(\varepsilon) \to C_n(\mathbb{C})$ by
$$
f_{\mathfrak A}(\zeta) \stackrel{def}= \mathfrak{A}(\zeta)({\mathcal P}_{\mathcal T} \Big(\mathcal{F}(\zeta))\Big)\,.
$$
The following lemma guarantees the existence of a mapping $ \mathfrak{A}$ that is suitable for our purposes.

\index{$\mathcal{U}(\varepsilon)$} \index{$f_{\mathfrak{A}}$}
\begin{lemm}\label{lemm5.2} There exists a holomorphic map
${\mathfrak A} : {\mathcal U}(\varepsilon) \to {\mathcal A}$ such
that the mapping $f(\zeta) \overset{\rm def}{=} {\mathcal P}_{\rm
sym} \, f_{\mathfrak A}  (\zeta)$, $\zeta \in {\mathcal
U}(\varepsilon)$, has the following properties:
\begin{equation}\label{eq5.9a}
f(K\zeta) = f(\zeta) \quad {\it for} \quad \zeta \in {\mathcal U}'
(\varepsilon) \stackrel{def}=\{\zeta \in \mathbb{C}:1-\varepsilon <|\zeta|<1+ \varepsilon \}   \,,
\end{equation}
and the family $f(it)$, $t \in [1,K]$, defines a closed path in $C_n
({\mathbb C}) \diagup {\mathcal S}_n$ in the free isotopy class
$\widehat{\Delta_n^{2\ell} \, b}$ for some $\ell \in {\mathbb Z}$.
\end{lemm}

\medskip

\noindent {\bf Proof of Lemma \ref{lemm5.2}.} 
By Lemma \ref{lem2.1a}
in a neighbourhood of each point $\zeta$
in $\mathbb{C}_+$ we may choose a smooth family $w_{\zeta}\in QC_{\infty}(0,n+1)$, such that $w_{\zeta}$   fixes $0$ and $\frac{1}{n}$, $[w_{\zeta}]=\mathcal{F}(\zeta)$, and
the equality ${\mathcal P}_{\mathcal T} ( \mathcal{F} (\zeta)) = e_n(w_{\zeta})$
holds.

The mapping $\varphi_{b,\infty}$ is a self-homeomorphism of $\mathbb{C} \setminus E_n^0 $. Its extension to $\mathbb{C}$ acts on $E_n^0$ by permutation:
\begin{equation}
\label{eq5.13}
\varphi_{b,\infty}(\frac{j}{n})=S_b(\frac{j}{n}), \, j=0,1,\ldots,n-1,
\end{equation}
for the permutation $S_b=\tau(b)$ that is related to the braid $b$ and maps the set $\{0,\frac{1}{n}, \ldots, \frac{n-1}{n}\}$ onto itself. Equation \eqref{eq5.13}
 is equivalent to the equation
$e_n(\varphi_{b,\infty})=S_b(e_n({\rm Id}))$.

Since $w_{\zeta}$ represents $\mathcal{F}(\zeta)$,
the mapping $w_{\zeta} \circ {\varphi}_{b,\infty}$ represents $\mathcal{F}(K\zeta)=  \varphi_{b,\infty}^*(\mathcal{F}(\zeta))$. Hence, the point
$e_n(w_{\zeta} \circ \varphi_{b,\infty})$
represents the class of
${\mathcal P}_{\mathcal T} ( \mathcal{F}(K \zeta))\,$
in the quotient $C_n(\mathbb{C})\diagup \mathcal{A}$     of $C_n(\mathbb{C})$ by the action of the group $\mathcal{A}$.

The locally defined family $\zeta\to e_n(w_{\zeta})$ does not depend on the local choice of the $\omega_{\zeta}$ and is holomorphic since it coincides with the holomorphic mapping $\zeta\to \mathcal{P}_{\mathcal{T}}(\mathcal{F}(\zeta))$. 
For $e_n(w_{\zeta})
 =\big(0,\frac{1}{n},\ldots,z_n(\zeta)\big)=\big(z_1(\zeta),\ldots,z_n(\zeta)\big)$
we have 
$e_n(w_{\zeta} \circ \varphi_{b,\infty})=
\big(z_{S(1)}(\zeta),\ldots,z_{S(n)}(\zeta)\big)$
for the permutation $S, \;S(j)=nS_b(\frac{j-1}{n})+1$ on $\{1,\ldots,n\}$, that is induced by $S_b$. 
Put $\,S(e_n(w_{\zeta}))=\,$ $S\big((z_1(\zeta),\ldots,z_n(\zeta))\big)=
(z_{S(1)}(\zeta),\ldots,z_{S(n)}(\zeta))$.
We obtained
\begin{equation}\label{eq5.13a}
e_n(w_{\zeta} \circ \varphi_{b,\infty})= S(e_n(w_{\zeta}))\,.
\end{equation}
Since $w_{\zeta}$ represents $\mathcal{F}(\zeta)$, and by  equation \eqref{eq5.12}
$w_{\zeta}\circ\varphi_{b,\infty}$ represents $\mathcal{F}(K\zeta)=\varphi_{b,\infty}^*(\mathcal{F}(\zeta))$, equation \eqref{eq5.13a} means that 
${\mathcal P}_{\mathcal T} ( \mathcal{F} (K \zeta))$ differs from 
$S({\mathcal P}_{\mathcal T} \Big(\mathcal{F}(\zeta))\Big)$ by the diagonal action of a M\"obius transformation depending on $\zeta \in \mathbb{C}_+$. This M\"obius transformation is determined by its action on the first two coordinates of ${\mathcal P}_{\mathcal T} (\mathcal{F}(K\zeta))$ and  $S({\mathcal P}_{\mathcal T}\Big( \mathcal{F} (\zeta))\Big)$.

The following lemma holds.
\begin{lemm}\label{lemm5.4} There exists a holomorphic mapping
${\mathfrak A} : {\mathcal U}(\varepsilon) \to {\mathcal A}$ such
that for each $\zeta \in {\mathcal U}' (\varepsilon)$ the first two
coordinates of the two $n$-tuples ${\mathfrak A} (\zeta)\Big(S ({\mathcal P}_{\mathcal T} \big( \mathcal{F} (\zeta))\big)\Big)$ and
${\mathfrak A}(K\zeta)\Big({\mathcal P}_{\mathcal T} \big( \mathcal{F} (K \zeta)\big)\Big)$
coincide on $\mathcal{U}'(\varepsilon)$.
\end{lemm}
We will prove this lemma after the proof of the present Lemma \ref{lemm5.2} is finished.

\smallskip

\noindent {\bf Continuation of proof of Lemma \ref{lemm5.2}}
Let ${\mathfrak A}$ be the mapping of Lemma \ref{lemm5.4}. We put $\;f_{\mathfrak A}  (\zeta) \stackrel{def}= \mathfrak{A}(\zeta)(\mathcal{P}_{\mathcal{T}}\Big( \mathcal{F}(\zeta))\Big)\,,\;$
$f_{\mathfrak A}: {\mathcal U}(\varepsilon) \to C_n(\mathbb{C})$. 
Since the first two coordinates of ${\mathfrak A} \big(K\zeta)\Big({\mathcal P}_{\mathcal T} \big( \mathcal{F} ( K\zeta)\big)\Big)$ and
${\mathfrak A} (\zeta)\Big(S({\mathcal P}_{\mathcal T} \big( \mathcal{F} (\zeta))\big)\Big)$ coincide on ${\mathcal U}(\varepsilon)$, 
equation \eqref{eq5.13a} implies
\begin{equation}\label{eq5.14}{\mathfrak A} (K\zeta)\Big({\mathcal P}_{\mathcal T} \big( \mathcal{F} (K \zeta)\big)\Big)={\mathfrak A} (\zeta)\Big(S\big({\mathcal P}_{\mathcal T} ( \mathcal{F} (\zeta))\big)\Big)
 \, ,
\quad \zeta \in{\mathcal U}'(\varepsilon) \, .
\end{equation}
Since the action of ${\mathfrak A} (\zeta)$
commutes with $S$, we have
\begin{equation}\label{eq5.15} {\mathfrak A} (K\zeta)\Big({\mathcal P}_{\mathcal T} \big( \mathcal{F}(K \zeta)\big)\Big)=S\Big({\mathfrak A} (\zeta)\big({\mathcal P}_{\mathcal T} ( \mathcal{F} (\zeta))\big)\Big)
 \, ,
\quad \zeta \in{\mathcal U}'(\varepsilon) \, .
\end{equation}
Equation \eqref{eq5.15} means that
$f_{\mathfrak
A} (K\zeta) = S (f_{\mathfrak A} (\zeta))$ for $\zeta \in{\mathcal
U}' (\varepsilon)$, hence,
\begin{equation}\label{eq5.16}
 {\mathcal P}_{\rm sym} \, f_{\mathfrak A} (K\zeta) = {\mathcal
P}_{\rm sym} \, f_{\mathfrak A} (\zeta) \, , \ \zeta \in{\mathcal
U}' (\varepsilon) \, .
\end{equation}
\noindent Equality \eqref{eq5.9a} is proved.

The second statement of
Lemma \ref{lemm5.2} is proved as follows. Take the smooth family $t\to w_{it} \in QC_{\infty}(0,n+1),\, t\in [1,K]$, of the beginning of the proof of Lemma \ref{lemm5.2}, for which $\mathcal{F}(it) =[w_{it}]$, and
$\mathcal{P}_{\mathcal{T}}( \mathcal{F}(it)) =     \mathcal{P}_{\mathcal{T}}([w_{it}])=e_n(w_{it})$.

We need to prove that $f(t)=\mathcal{P}_{\rm sym}(\mathfrak{A}(it)e_n(w_{it})),\, t\in [1,K],$ represents $\hat b$ up to a power of a Dehn twist.
By equation \eqref{eq5.15} the equality
$$\mathfrak{A}(K i t)\big(e_n(w_{Kit})\big)=S \big(\mathfrak{A}(it) (e_n(w_{it})\big)$$
holds. The mapping $\tilde{f}(it)\stackrel{def}= \mathfrak{A}(it)\big(e_n(w_{it})\big)=e_n\big(\mathfrak{A}(it)(w_{it})\big),\, t \in [1,K],$ is a lift to $C_n(\mathbb{C})$ of the mapping $f(it)=\mathcal{P}_{\rm sym}\big(\mathfrak{A}(it)e_n(w_{it})\big)= {\rm ev}_{E_n^0}\big(\mathfrak{A}(it)( w_{it})\big)$.
Hence, $t\to \mathfrak{A}(it)( w_{it})$ is a parameterizing isotopy of the geometric braid $t\to f(it)$. Since $[w_{iK}]=[w_i\circ \varphi_{b,\infty}]$, also the equality
\begin{align*}
[\mathfrak{A}(iK)(w_{iK})]=[ \mathfrak{A}(i) (   w_i\circ \varphi_{b,\infty})]
\end{align*}
holds.
Since
\begin{align*}
e_n\big(\mathfrak{A}(iK) \circ w_{iK}\big)\quad \;&=\mathfrak{A}(iK) \big (e_n( w_{iK})\big)= \mathfrak{A}(iK)\big(\mathcal{P}_\mathcal{T}(\mathcal{F}(iK))\big)\,,\\
e_n(\mathfrak{A}(i)  \circ  w_i\circ \varphi_{b,\infty})&= \mathfrak{A}(i) (S(e_n(w_i)))\;\,=S (\mathfrak{A}(i)(\mathcal{P}_\mathcal{T}(\mathcal{F}(i))\,,
\end{align*}
equation \eqref{eq5.14}   and Lemma \ref{lemm2.6} imply that the Teichm\"uller equivalent mappings $\mathfrak{A}(iK)(w_{iK})$ and $\mathfrak{A}(i)( w_{i}\circ \varphi_{b,\infty})$ are isotopic. Hence, the class of $\varphi_{b,\infty}$ is the mapping class in $\mathfrak{M}(\mathbb{P}^1;\infty,E_n^0  )$ of the geometric braid $t\to f(it),\,t\in [1,K]$. 
Hence, the conjugacy classes of the braid $b$ and of the braid represented by $t\to f(it),\,t\in[1,K]$ differ by a power of a Dehn twist.
Lemma \ref{lemm5.2} is proved. \hfill $\Box$

\medskip

\noindent {\bf Proof of Lemma \ref{lemm5.4}.}
Write
$$
{\mathcal P}_{\mathcal T} (\mathcal{F}(\zeta) )= \big(z_1 (\zeta) , z_2 (\zeta) , z_3 (\zeta) , \ldots,\,z_n (\zeta)\big)
$$
with $z_1 (\zeta) \equiv 0$, $z_2 (\zeta) \equiv \frac1n$.
For the proof of the lemma we have to find an affine mapping $\mathfrak{A} \in \mathcal{A}$ so that the following condition is satisfied
\begin{align}
\label{eq5.17}
{\mathfrak A} (K\zeta)\;\Big(\,0\,\Big)\quad \,\, & \equiv \,{\mathfrak A} (K\zeta)(z_1(K\zeta))\,  =\, {\mathfrak A} (\zeta) (z_{S(1)}(\zeta)) \, , \nonumber \\
{\mathfrak A} (K\zeta) \left( \frac1n \right)\quad \,&\equiv \,  {\mathfrak A} (K\zeta)(z_2(K\zeta))\,  = \,{\mathfrak A}
(\zeta) (z_{S(2)} (\zeta)) ,
\quad \zeta \in {\mathcal U}'(\varepsilon) \, .
\end{align}
We write ${\mathfrak A} (\zeta)$ in the form ${\mathfrak A}
(\zeta) \,( z) = \frac{z-b(\zeta)}{a(\zeta)}$, $z \in {\mathbb C}$,
for holomorphic functions $a$ and $b$ in $\mathbb{C}_+$ 
with $a \ne 0$ in $\mathbb{C_+}$. 
For a number $K>1$ and any function $q$ on ${\mathcal
U}(\varepsilon)$ we use the notation $q^K (\zeta)$, $\zeta \in
{\mathcal U}'(\varepsilon)$, for the function $q^K (\zeta) =
q(K\zeta)$. The equations \eqref{eq5.17} can be written on ${\mathcal U}' (\varepsilon)$ as
\begin{align}
\label{eq5.18}
\frac{z_{S(1)}-b}{a} &= \frac{ - b^K}{a^K} \, , \nonumber \\
\frac{z_{S(2)} - b}{a}&= \frac{\frac1n - b^K}{a^K} \,.
\end{align}
Put
$$
\chi \overset{\rm def}{=} \frac{1}{n (z_{S(2)} - z_{S(1)})}\,.
$$
$\chi$ is
an analytic function on  ${\mathcal
U}'(\varepsilon)$, $\chi \ne 0$ on  ${\mathcal
U}'(\varepsilon)$. The equations \eqref{eq5.18} can be rewritten as
\begin{align}
\label{eq5.19}
\frac{a^K}{a} &= \chi \, , \nonumber \\
\frac{b^K}{a^K} - \frac ba &= -\frac{z_{S(1)}}{a} \, , \quad \zeta
\in {\mathcal U}' (\varepsilon) \, .
\end{align}
The first equation leads to a Second Cousin Problem for the coefficient $a$ on the annulus $\mathbb{C}_+\diagup \Lambda_K$ (see Appendix A). Indeed, let $\varepsilon >0$ be small compared to $K$. Cover the annulus $\mathbb{C}_+\diagup \Lambda_K$ by the open sets
\begin{align}\nonumber
\mathcal{U}_1 &= \{ \zeta \in {\mathbb C}_+ : 1 -
\varepsilon < \vert \zeta \vert < 1+ 3 \varepsilon \}\diagup \Lambda_K\;,\nonumber\\
\mathcal{U}_2  &= \{ \zeta \in {\mathbb C}_+ : 1 +2
\varepsilon < \vert \zeta \vert < 1+ 5\varepsilon \}\diagup \Lambda_K\;,\; \mbox{and}\nonumber\\
\mathcal{U}_3 &= \{ \zeta \in {\mathbb C}_+ : 1 +4
\varepsilon < \vert \zeta \vert < (1+ \varepsilon )K\}\diagup \Lambda_K\;.
\end{align}
Then $\;\mathcal{U}_1 \cap \mathcal{U}_2 \cap \mathcal{U}_3 =\emptyset$, and $\mathcal{U}_3\cap \mathcal{U}_1= \mathcal{U}'(\varepsilon)\diagup \Lambda_K$. On $\mathcal{U}_{3,1}=\mathcal{U}_3 \cap \mathcal{U}_1$ we take the holomorphic transition function $g_{3,1}(\zeta\diagup \Lambda_K)=\chi(\zeta), \,\zeta \in \mathcal{U}'(\varepsilon)$,  on all other intersections of covering sets we take the holomorphic functions that are identically equal to one. We obtained a Cousin II distribution. Since the intersections of two covering sets are simply connected and an annulus is a Stein manifold, the Second Cousin Problem has a solution. This means, there exist nowhere vanishing holomorphic functions $g_j$ on  $\mathcal{U}_j$ such that $g_{j,k}= \frac{g_j}{g_k}$. 
Cover $\mathcal{U}_{\varepsilon}$ by lifts $\tilde{\mathcal{U}}_j$ of $\mathcal{U}_j$, $j=1,2,3$.
Let $\tilde{ g}_j$ be the lift of $g_j$ to $\mathcal{U}_j$
The function $a$ on $\mathcal{U}(\varepsilon)$ which is equal to $\tilde{g}_j$ on $\tilde{\mathcal{U}}_j$, $j=1,2,3,$ satisfies the first equation of \eqref{eq5.19}.

After we found a function $a$  satisfying the first equation, we are looking for a function $\frac{b}{a}$ which satisfies the second equation. This leads to a First Cousin Problem on $A$ which is solvable. The lemma is proved. \hfill $\Box$

\medskip

\noindent {\bf End of proof of Proposition \ref{prop5.1}.}
To prove the Proposition \ref{prop5.1} we consider the mapping
${\mathfrak a}_{\ell} (\zeta) (f(\zeta))$, $\zeta \in {\mathcal
U}(\varepsilon)$. Here $f$ is the mapping of Lemma \ref{lemm5.2} and
\begin{equation}\label{eq5.10}
{\mathfrak a}_{\ell} (\zeta)(z) = e^{\frac{2\pi i \ell}{\log K}
\cdot \log \frac\zeta i} \cdot z \, , \quad \zeta \in {\mathcal
U}(\varepsilon) \, , \quad z \in {\mathbb C} \, .
\end{equation}
Equation \eqref{eq5.10} defines a holomorphic mapping from $A = {\mathbb C}_+
\diagup \Lambda_K$ to $C_n ({\mathbb C}) \diagup {\mathcal S}_n$.
When $\zeta$ ranges over $[i,iK]$ the function $\log \frac\zeta i$
ranges over $[1,\log K]$, and ${\mathfrak a}_{\ell} (\zeta)$, $\zeta
\in [i,iK]$, defines the $\ell$-th power of the full twists about the circle $[i,iK] \diagup \Lambda_K$. Hence, 
for a suitable integer number $\ell$ the path ${\mathfrak a}_{\ell}
(\zeta)(f(\zeta))$, $\zeta \in [i,iK]$, represents the free isotopy
class $\hat b$. 
The conformal module $m(A) = m({\mathbb C}_+ \diagup
\Lambda_K)$ equals
$$
\frac\pi2 \frac1{\log \frac K2} \, =  \frac\pi2 \frac1{h(\hat b)}.
$$

\noindent Proposition \ref{prop5.1} is proved in the hyperbolic case. \hfill $\Box$

\bigskip

\chapter[Pur braids. Irreducible components. Proof of the Main Theorem.]
{
Reducible pure braids. Irreducible nodal components, irreducible braid components,  and the proof of the Main Theorem.}
\label{chapter6}
\setcounter{equation}{0}

\bigskip

\noindent The main purpose of this chapter is the proof of the Main Theorem for
conjugacy classes of reducible pure braids.
The idea is to reduce the proof to the case of irreducible braids and mapping classes. We will describe a decomposition of reducible elements of
$\widehat{\mathfrak M} ({\mathbb P}^1 ; \{\infty\}\cup  E_n) $ into
irreducible nodal components (See Section \ref{sec:6.1}). The elements of  $\widehat{\mathfrak M} ({\mathbb P}^1 ; \{\infty\}\cup E_n)$ can be reconstructed from the irreducible nodal components up to a family of commuting powers of Dehn twists.
On the other hand we provide a decomposition of conjugacy classes of reducible pure braids into irreducible braid components. The
conjugacy class of a reducible pure braid can be recovered from its irreducible braid components.

In Section \ref{sec:6.3} we will establish the relation between the
irreducible braid components of the conjugacy class of pure braids and the irreducible nodal components of the conjugacy class of mapping classes associated to the pure braid.

In Section \ref{sec:6.4} we prove that the conformal module of a conjugacy class $\hat b$ of pure braids is equal to the minimum of the conformal modules of the irreducible braid components, and the entropy of the class $\hat b$ is equal to the maximum of the entropies of the irreducible nodal components of the mapping class associated to $\hat b$. The Main Theorem for reducible pure braids is proved by applying the version of the Main Theorem for irreducible braids to the irreducible braid components and the related irreducible nodal components.

\medskip

\section {Irreducible nodal components. Pure braids.}\label{sec:6.1}
Throughout this chapter 
$b \in {\mathcal B}_n$ will be a reducible pure braid and 
$\mathfrak{m}_b \in {\mathfrak M} (\overline{\mathbb{D}} ; \partial \mathbb{D} \cup  E_n)$ its mapping class. Here $E_n\subset \mathbb{D}$ is a set consisting of $n$ points.
We will describe the decomposition of the conjugacy class $\widehat{\mathfrak{m}_b}$ into irreducible components.
First, we represent the
mapping class
${\mathfrak m}_{b,\infty}=\mathcal{H}_{\infty}(\mathfrak{m}_b) \in {\mathfrak M} ({\mathbb{P}^1} ; \{\infty\}\cup E_n)$ 
by a self-homeomorphism  $\varphi_{b,\infty}$ of $\mathbb{P}^1$ 
which is the identity outside the unit disc $\mathbb{D}$. We assume that $\varphi_{b,\infty}$ is completely reduced by
an admissible system of curves  ${\mathcal C} = \{ C_1 ,
\ldots , C_k\}$  in ${\mathbb D}
\backslash E_n \subset \mathbb{C}\backslash E_n $. In particular,
$\varphi _{b,\infty}$ leaves the union $\underset{C \in\mathcal{C}}{\bigcup} C$
invariant, and also leaves the complement ${\mathbb P}^1 \backslash
\underset{C \in \mathcal{C}}{\bigcup}C$ invariant. Notice that the decomposition into irreducible components in general depends on the isotopy class of the admissible system  
$\mathcal{C}$.

\smallskip

The Jordan Curve Theorem induces a partial order on the set of
connected components of ${\mathbb P}^1 \backslash
\underset{C\in \mathcal{C}}{{\bigcup}} C$.
The partial order is
described as follows.
\index{Jordan Curve Theorem}

For simplicity of notation we put $E_n'= E_n \cup \infty$. We will label the connected components of ${\mathbb P}^1 \backslash
\underset{C \in \mathcal{C}}{\bigcup} C$ by $S^{\ell,j}$, were the first digit indicates the number of the generation, the second digit indicates the label of the component among components of the same generation.
There is only one connected component of first generation, 
denoted by $S^{1,1}$. This is the outermost \index{$S^{1,1}$}  connected component of ${\mathbb
P}^1 \backslash \underset{C\in \mathcal{\mathcal{C}}}{\bigcup} \, C$, i.e. the component which contains the point $\infty$. 
We say that the exterior boundary of $S^{1,1}$ is empty.
The union of all boundary components of $S^{1,1}$ is denoted by
$\partial_{\mathfrak I} \, S^{1,1} = \partial S^{1,1}$ and called 
the interior boundary $\partial_{\mathfrak I} \, S^{1,1} = \partial S^{1,1}$ of $S^{1,1}$. The components of $\partial_{\mathfrak I} \, S^{1,1}$ are called the curves of generation 1.

By the Jordan Curve Theorem each connected component of  the interior boundary $\partial_{\mathfrak I} \, S^{1,1}$ bounds a disc contained
 in $\mathbb{D}$. Denote the discs 
by $\delta^{2,j}, \; j=1,\ldots, k_2$. For each $j$ there is exactly one connected component $S^{2,j}$ of ${\mathbb C}
\backslash \underset{C \in \mathcal{C}}{\bigcup} \, C$ which is contained in $\delta^{2,j}$ and has a
boundary component in common with $S^{1,1}$. (For short, $S^{2,j}$ is adjacent to a boundary component of $S^{1,1}$.) 
We call this boundary component
the exterior boundary component of $S^{2,j}$ and denote it by
$\partial_{\mathfrak E} \, S^{2,j}$, $j = 1,\ldots , k_2$.
\index{$\partial_{\mathfrak E} \, S^{2,j}$} The set
$\partial_{\mathfrak I} \, S^{2,j} = \partial S^{2,j} \backslash
\partial_{\mathfrak E} \, S^{2,j}$ is \index{$\partial_{\mathfrak I} \, S^{2,j}$} called the interior boundary
of $S^{2,j}$.
The components $S^{2,j}$ are called the components of
second generation. The exterior boundaries $\partial_{\mathfrak E} \, S^{2,j}$ of the components $S^{2,j}$, $j = 1,\ldots , k_2$, of generation $2$ are called the curves of generation $2$.

\smallskip

The connected components of ${\mathbb P}^1 \backslash
\underset{j=1}{\overset{k}{\bigcup}} C_j$ of generation $\ell$
are defined by induction as follows.
Consider the union of all connected components of ${\mathbb P}^1
\backslash \underset{j=1}{\overset{k}{\bigcup}} C_j$ of generation
not exceeding $\ell - 1$. Take its closure
\index{$Q_{\ell}$}
\begin{equation}\label{eq6.0}
\overline{Q_{\ell -1}}\stackrel{def}= \overline{\underset{1\leq \ell' \leq \ell - 1}{\bigcup}\; \underset{1
\leq j \leq k_{\ell'} }{\bigcup} S^{\ell' , j}}\,,
\end{equation}
which is the
closure of a domain $Q_{\ell - 1} \subset {\mathbb P}^1\, .\,$
$Q_{\ell - 1}$ is the union of all components of generation not
exceeding $\ell-1$ and the exterior boundaries of all components of
generation between 2 and $\ell - 1$.
The connected components of
${\mathbb P}^1 \backslash \underset{j=1}{\overset{k}{\bigcup}} C_j$
which share a boundary component with $Q_{\ell - 1}$ 
are called the components of generation $\ell$
and are labeled by $S^{\ell , j}$, $j=1,\ldots , k_{\ell}$. Each
$S^{\ell ,j}$ has exactly one boundary component in common with
$Q_{\ell - 1}$. We call this part of the boundary of $S^{\ell ,j}$
the exterior boundary of $S^{\ell ,j}$ and denote it by
$\partial_{\mathfrak E} \, S^{\ell , j}$. The exterior boundaries of the
components $S^{\ell , j}$ are also called the curves of generation $\ell$.
The remaining set
$\partial S^{\ell , j} \backslash
\partial_{\mathfrak E} \, S^{\ell,  j}$ is denoted by
$\partial_{\mathfrak I} \, S^{\ell , j}$ and is called the interior
boundary of $S^{\ell , j}$. If some component $S^{\ell , j}$ of generation $\ell$ coincides with the disc $\delta^{\ell , j}$, its interior boundary is empty.
\index{$S^{\ell ,j}$} \index{$\partial_{\mathfrak E} \, S^{\ell , j}$} \index{$\partial_{\mathfrak I} \, S^{\ell , j}$} \index{$\delta^{\ell , j}$} \index{$k_{\ell}$}  \index{$k(\ell,i)$}

\smallskip

The process terminates after we obtained sets of generation $N$ for
some finite number $N$. The components $S^{N,j}$ are the innermost components, they are equal to the discs $\delta^{N,j}$, in other words, their interior boundary is empty.

Each of the curves in ${\mathcal C}$
is the exterior boundary of a single
connected component of ${\mathbb P}^1 \setminus \bigcup_{C \in \mathcal{C}}C$.

The mapping $\varphi_{b,\infty}$ fixes each connected component $S^{\ell,j}$.
Indeed, it fixes $S^{1,1}$ since it fixes $\infty \in S^{1,1}$. Hence it fixes the interior boundary of $S^{1,1}$ and therefore, it fixes the union of the discs $\delta^{2,j}$.
Further, since the system of curves $\mathcal{C}$ is admissible, each disc $\delta^{2,j}$ contains a point of $E_n$. Hence, since $b$ is a pure braid, the mapping $\varphi_{b,\infty}$ fixes the point. Hence, it fixes each disc $\delta^{2,j}$, and therefore it fixes the component $S^{2,j}$ whose exterior boundary component is the boundary of the disc $\delta^{2,j}$. By an induction on the number of generation we see that
$\varphi_{b,\infty}$ fixes each connected component $S^{\ell,j}$.

Let
$$
w : {\mathbb C} \setminus  E_n'  \to Y
$$
be a continuous surjection onto a nodal surface $Y$ associated to
the isotopy class of the curve system ${\mathcal C}$ in $\mathbb{P}^1$
(see Section \ref{sec:2.4}). Notice that the nodal surface $Y$ has punctures.
The mapping $w$ maps each component $\,S^{\ell,j} \setminus E_n' \,$
of $\,{\mathbb P}^1 \setminus ( E_n' \cup \bigcup C_j)\,$
homeomorphically onto a part $Y^{\ell , j}$ of $Y$. \index{$Y^{\ell
, j}$}(Recall that a part of a nodal surface $Y$ with set of nodes
${\mathcal N}$ is a connected component of $Y \backslash {\mathcal
N}$.) The correspondence between the components $S^{\ell , j}
\setminus E_n'$ and the parts $Y^{\ell , j}$ is a bijection.
\index{nodal surface}    \index{nodal surface ! part of}     \index{$Y^{\ell , j}$}
\index{$\mathfrak{m}_{b,{\odot}}$} \index{$\widehat{\mathfrak{m}_{b,{\odot}}}$}

We conjugate the restriction $\varphi_{b,\infty}|\mathbb{P}^1\setminus (E_n'\cup \mathcal{C})$ of the self-homeomorphism $\varphi_{b,\infty}$ with the inverse of $w$. 
We obtain a self-homeomorphism of $Y\setminus \mathcal{N}$. Its isotopy class of self-homeomorphisms of the nodal surface $Y$ is denoted by $\mathfrak{m}_{b,{\odot}}$  and is called the isotopy class of self-mappings of the nodal surface $Y$ determined
by ${\mathfrak m}_{b,\infty}$ and the isotopy class of the admissible system $\mathcal{C}$. Let $\widehat{{\mathfrak m}_{b,{\odot}}}$ be the conjugacy class of
${\mathfrak m}_{b,{\odot}}$.

Let $Y^c$ be the compact nodal surface obtained by filling the punctures of $Y$. The surjection $w$ extends to a continuous surjection $w:\mathbb{P}^1\to Y^c$ which by an abuse of notation we denote again by $w$. The elements of the class $\mathfrak{m}_{b,{\odot}}$ extend across the punctures of $Y$ to self-homeomorphisms of $Y^c$ with set of distinguished points $\mathcal{N}\cup w(E_n')$. By an abuse of notation we will also denote the classes of these extensions by $\mathfrak{m}_{b,{\odot}}$, and   $\widehat{{\mathfrak m}_{b,{\odot}}}$ respectively.

More detailed,
the mapping $\varphi_{b,\infty}$ fixes each $S^{\ell,j}\setminus E_n' $ setwise, hence each mapping in ${\mathfrak m}_{b,{\odot}}$ fixes each nodal component $Y^{\ell,j}$ setwise.  For each $(\ell,j)$ the restriction  ${\mathfrak m}_{b,{\odot}}^{\ell,j}\stackrel{def}={\mathfrak m}_{b,{\odot}}|Y^{\ell,j}$ is obtained
in the following way. Put ${E'}^{\ell,j}\stackrel{def}=E'\cap S^{\ell,j}$. We conjugate the restriction of $\varphi_{b,\infty}$ to $S^{\ell,j}\setminus {E'}^{\ell,j}$ by the inverse of the restriction $w^{\ell,j}\stackrel{def}=w|S^{\ell,j}\setminus {E'}^{\ell,j} $. The class  ${\mathfrak m}_{b,{\odot}}^{\ell,j}$ is the mapping class of the conjugated mapping which is a self-homeomorphism of  $Y^{\ell,j}$. The conjugacy class $\widehat{{\mathfrak m}_{b,{\odot}}}$ is called the nodal conjugacy class
associated to $\widehat{{\mathfrak m}_{b,\infty}}$ and the admissible system of curves $\mathcal{C}$.

The elements of ${\mathfrak m}_{b,{\odot}}|Y^{\ell,j}$ extend continuously across the punctures of $Y^{\ell,j}$. The extensions are self-homeomorphisms  of the closed Riemann surface $(Y^{\ell,j})^c \cong \mathbb{P}^1$ with distinguished points $w({E'}^{\ell,j}) \cup (\mathcal{N} \cap  ({Y^{\ell,j})^c})$. The set of extensions is denoted by the same letter  ${\mathfrak m}_{b,{\odot}}^{\ell,j}$.
The conjugate of $\varphi_{b,\infty}\mid S^{\ell,j}\setminus {E'}^{\ell,j}$ by $(w^{\ell,j})^{-1}$
represents the class
${\mathfrak m}_{b,{\odot}}^{\ell,j}$ on $Y^{\ell,j}$, and, hence,  $\varphi_{b,\infty}\mid S^{\ell,j}\setminus {E'}^{\ell,j}$ represents the conjugacy class $\widehat{{\mathfrak m}_{b,\odot}^{\ell,j}}$.

The conjugacy classes $\widehat{{\mathfrak m}_{b,\odot}^{\ell,j}}$  will be called the irreducible nodal components of the class
$\widehat{{\mathfrak m}_{b,\infty}}$ with respect to the system $\mathcal{C}$.
\index{nodal component}
The conjugacy classes $\widehat{{\mathfrak m}_{b,\odot}^{\ell,j}}$  of the restrictions $ {\mathfrak m}_{b,\odot}^{\ell,j}$ depend only on the conjugacy class $\widehat{{\mathfrak m}_{b,\infty}}$ and the isotopy class of the curve system $\mathcal{C}$.
Notice, that the irreducible
nodal components determine the class $\widehat{{\mathfrak
m}_{b,\infty}}$ only up to products of powers of some Dehn twists.
(See Section \ref{sec:6.4}, the remark after the proof of Proposition \ref{prop7.2}).
But they are the only objects related to a decomposition of mappings
which depend only on the isotopy class of the curve system
${\mathcal C}$ and on the conjugacy class $\widehat{{\mathfrak
m}_{b,\infty}}$.

\medskip

\section{The irreducible braid components. Pure braids.}\label{sec:6.2}
In this section we will describe the decomposition of reducible braids into
irreducible braid components. The decomposition depends again on the isotopy class of an admissible system of curves that completely reduces the mapping class associated to the braid.

\index{braid ! irreducible braid components}

\smallskip

We start with defining a few objects. This will be done here for the case of arbitrary (not necessarily pure) braids.
A tube in the ``infinite'' cylinder $[0,1] \times {\mathbb
C}$ is the image of a ''finite''  cylinder $[0,1] \times \overline{\mathbb D}$ under a diffeomorphism of $[0,1] \times {\mathbb C}$ onto itself which preserves each fiber $\{t\} \times \mathbb{C},\, t\in [0,1]$.
Let $\Omega^{n_2}=\{\Omega_1,\ldots,\Omega_{n_2}\}$ be  an unordered $n_2$-tuple of pairwise disjoint closed topological discs in the complex plane, and $E_{n_1}=\{z_1,\ldots,z_{n_1}\}$ an unordered $n_1$-tuple of points in $\mathbb{C}\setminus \cup \Omega_j$. A tubular geometric $(n_1,n_2)$-braid with base point $(E_{n_1},\Omega^{n_2})$ consists of the following two objects. The first object is a collection of $n_2$ mutually disjoint tubes in the  cylinder $[0,1] \times {\mathbb
C}$, each of which intersects the fibers $\{ 0 \} \times \mathbb{C}$ and $\{ 1\} \times \mathbb{C}$  along a copy of $\Omega^{n_2}$. The second object is a geometric $n_1$-braid with base point $E_{n_1}$ whose strands do not meet the $n_2$ tubes.
\index{braid ! tubular geometric} 
\index{braid ! tubular geometric $(n_1,n_2)$-braid}

Take a point $E_{n_1} \in C_{n_1}(\mathbb{C})\diagup \mathcal{S}_{n_1}$
and a point $\mathbold{E}_{n_2} \in C_{n_2}(\mathbb{C})
\diagup \mathcal{S}_{n_2}$
such that $E_{n_1} \cup \mathbold{E}_{n_2} \in C_{n_1+n_2}(\mathbb{C})
\diagup \mathcal{S}_{n_1+n_2}$.
A geometric fat braid (of type $(n_1,n_2)$) with base point $E_{n_1}\cup \mathbold{E}_{n_2}$ is a geometric $(n_1 +n_2)$-braid such that all strands with initial point in $\mathbold{E}_{n_2}$ have also their endpoints in $\mathbold{E}_{n_2}$ and are declared to be fat. The strands with initial point in $E_{n_1}$ will be called ordinary strands. A tubular geometric braid is associated to a fat braid if there is a one-to-one correspondence between the tubes and the fat strands so that each fat strand is a deformation retraction of the corresponding tube.
\index{braid ! geometric fat braid  of type $(n_1,n_2)$} \index{braid ! geometric tubular-fat braid}

A fat braid is an isotopy class of geometric fat braids with fixed base point. \index{braid ! fat braid}
Two geometric fat braids are called isotopic if they are isotopic as geometric braids and the isotopy moves the fat strands to fat strands and the ordinary strands to ordinary strands.
A geometric tubular fat braid is a tubular geometric braid
such that forgetting the tubular strands gives a geometric fat braid.
\index{braid ! geometric tubular-fat}

We consider now again pure braids.
Take a reducible pure braid $b\in \mathcal{B}_n$ with base point $E_n^0$, and describe the decomposition into irreducible braid components. As in section \ref{sec:6.1} we consider a mapping
$\varphi_{b,\infty} \in {\mathfrak m}_{b,\infty} \in {\mathfrak
M} (\mathbb{P}^1 ; \{\infty\} \cup E_n^0)$ that is the identity outside the unit disc and is completely reduced by an admissible system ${\mathcal C}$ of
curves contained in the unit disc. (Notice that in general the decomposition into irrecucible braid components depends on the isotopy class of the admissible system $\mathcal{C}$.)
Replacing $\varphi_{b}$ by an isotopic mapping which differs from the original mapping only in small neighbourhoods of the admissible curves we may assume that
$\varphi_{b,\infty}$ fixes each curve $C $ pointwise.

\smallskip

Consider a path $\varphi_t \in {\rm Hom}^+
(\mathbb{P}^1 ; \mathbb{P}^1\setminus {\mathbb D})$, $t \in [0,1]$,
which joins $\varphi_{b,\infty}$ with the identity. More detailed, $\varphi_t$
is a continuous family of self-homeomorphisms of ${\mathbb
P}^1$ which fix the complement $ \mathbb{P}^1\setminus {\mathbb D}$ of the unit disc pointwise, such
that $\varphi_0 = {\rm id}$ and $\varphi_1 = \varphi_{b,\infty}$.
The evaluation map
\begin{equation}\label{eq6.2a}
[0,1] \ni t \to {\rm ev}_{E_n^0} \, \varphi_t = \left\{ \varphi_t
(0) , \ldots , \varphi_t \left( \frac{n-1}n \right) \right\} \in C_n
({\mathbb C}) \diagup {\mathcal S}_n \,
\end{equation}
gives a geometric braid that is contained in the cylinder $[0,1] \times {\mathbb D}$ (i.e. ${\rm ev}_{E_n^0} \, \varphi_t \in {\mathbb D}$ for each $t$), has base point
$E_n^0 = \varphi_0 (E_n^0) = \varphi_1 (E_n^0) \, ,$ and represents $b$.
The family $\varphi_t,\, t \in[0,1]$ is a parameterizing isotopy of the geometric braid \eqref{eq6.2a}.
We will write the geometric braid \eqref{eq6.2a} 
as continuous mapping $g:[0,1]\to C_n(\mathbb{C})\diagup{\mathcal S}_n$,
\begin{equation}\label{eq6.1a}
g(t)=\varphi_t (E_n),\, t \in [0,1]\,,
\end{equation}
and label the coordinate functions of $g$ 
so that for $z'\in E_n$ the coordinate function $g_{z'}$ of $g$ has initial point $g_{z'}(0)=g_{z'}(1)=z'$.

We will associate to the parameterizing isotopy tubular geometric braids and geometric fat braids. The classes obtained from the fat braids by isotopy and conjugation are the irreducible braid components. This is done as follows.

\bigskip

\noindent {\bf The outermost braid component  $\widehat{ \mathbold{b}(1,1)}$ of a pure braid $b$.}
Recall that we regard a point in $C_n(\mathbb{C})$ sometimes as ordered subset of $\mathbb{C}$, and sometimes as a point in $\mathbb{C}^n$.

\index{braid ! braid component}

We will associate now a fat braid $\mathbold{b}(1,1)$ to the outermost component $S^{1,1}$.
Recall that ${E'}^{1,1} \stackrel{def}= E'_n \cap S^{1,1}$ \index{$E^{1,1}$} is the
set of distinguished points contained in the component $S^{1,1}$. Let
$C^{2,j}$, $j = 1,\ldots , k_2$, be the interior boundary
components of $S^{1,1}$, labeled so that $C^{2,j}$ is the exterior
boundary of the component $S^{2,j}$ of $\mathbb{P}^1\setminus \cup_{C \in\mathcal{C}} C$, $j = 1,\ldots , k_2$. Denote by
$\delta^{2,j} \subset {\mathbb D}$ the \index{$\delta^{2,j}$}
topological disc bounded by $C^{2,j}$.

The tubular geometric braid associated to the parameterizing isotopy $\varphi_t$ and the component $S^{1,1}$ is
\begin{equation}\label{eq6.2}
 \{t\} \times \varphi_t \left(E^{1,1} \cup \bigcup_{j=1}^{k_2} \,
\overline{\delta^{2,j}} \right) \, , \quad t \in [0,1] \,.
\end{equation}
For each tube $ \{t\} \times \varphi_t (\overline{\delta^{2,j}}),\, t \in [0,1],$ we consider a strand contained in it and call it a fat strand. More detailed,
for each $j$ we pick a point
$z^{2,j} \in \delta^{2,j} \cap E_n$ and denote it by the fat letter $\mathbold{z}^{2,j}$. Such a point exists by the
definition of an admissible system of curves and an induction
argument. Since the braid is pure we have $\varphi_{b,\infty} (\mathbold{z}^{2,j}) =
\mathbold{z}^{2,j}$. The respective fat strand is $\varphi_t(\mathbold{z}^{2,j}), \, t \in [0,1]$.
Put $\mathbold{E}^{1,1} = \{\mathbold{z}^{2,1} , \ldots , \mathbold{z}^{2,k_2}\}$.
\index{$\mathbold{E}^{1,1} $} Note that the choice of the points in $\mathbold{E}^{1,1}$ may not be unique.
We associate to $S^{1,1}$ the set of points ${E}^{1,1} \cup \mathbold{E}^{1,1}$.

The geometric fat braid associated to $S^{1,1}$ and the parameterizing isotopy equals
\begin{equation}\label{eq6.1}
\varphi_t (E^{1,1} \cup \mathbold{E}^{1,1}) \, , \quad t \in
[0,1] \,.
\end{equation}

The strands with initial point in  $E^{1,1}$ are the ordinary strands, and the strands with initial point in  $\mathbold{E}^{1,1}$ are the fat strands.
The geometric fat braid \eqref{eq6.1} is obtained from the geometric tubular braid \eqref{eq6.2} by taking a deformation retract of each tube to a fat strand.

Let $n(1,1)$ be the number of points of $E^{1,1} \cup {\mathbold{E}^{1,1}}$.
We will write the geometric fat braid \eqref{eq6.1} as \index{$g^{1,1}$} mapping
\begin{equation}\label{eq6.3}
\mathbold{g}^{1,1} : [0,1] \to C_{n(1,1)} ({\mathbb C})\diagup{\mathcal S}_{n(1,1)} \, , \quad \mathbold{g}^{1,1}
(0) =  \mathbold{g}^{1,1} (1)= E^{1,1} \cup {\mathbold{E}^{1,1}} \,.
\end{equation}
If needed, the ''strands'' of $\mathbold{g}^{1,1} $ 
will be labeled by $ E^{1,1} \cup {\mathbold{E}^{1,1}}$, so that for $z'\in E^{1,1} \cup {\mathbold{E}^{1,1}}$ the strand corresponding to
$ \mathbold{g}^{1,1}_{z'}$ has initial point  $ \mathbold{g}^{1,1}_{z'}(0)=z'$.
The fat strands are those with initial points in ${\mathbold{E}^{1,1}}$.

We define the fat braid $\mathbold{b}(1,1)$ with base point $E^{1,1} \cup \mathbold{E}^{1,1}$
as the isotopy class of the geometric fat braid \eqref{eq6.1} (equivalently, \eqref{eq6.3}) with given base point. The fat strands are those with initial point in  $\mathbold{E}^{1,1}$.
Notice that $\mathbold{b}(1,1)$ is obtained from the braid $b$ by
discarding all strands with initial point not in $E^{1,1} \cup
{\mathbold{E}}^{1,1}$ and indicating the set of fat strands.

Finally we take the conjugacy class $\widehat{\mathbold{b}(1,1)}$ of the fat braid
$\mathbold{b}(1,1)$. This is the irreducible braid component of $\hat b$ associated to $b$ and the component $S^{1,1}$. The class $\widehat{\mathbold{b}(1,1)}$ depends only on the conjugacy class $\hat b$ and on the isotopy class of the curve system $\mathcal{C}$.

See Figure \ref{fig6.1} for an example of the geometric fat braid $ \mathbold{g}^{1,1}$
and the respective tubular geometric braid. In the
figure the set $E^{1,1}$ consists of a single point. The strand with this initial point is the only ordinary strand.
There are two
components of ${\mathbb C} \backslash \underset{C \in {\mathcal
C}}{\bigcup} C$ of generation 2, $S^{2,1}$ and $S^{2,2}$. The
respective distinguished points are $\mathbold{z}^{2,1}$ and $\mathbold{z}^{2,2}$. These are the points in
$\mathbold{E}^{1,1}$. The strands with these initial points are the fat strands of $\mathbold{g}^{1,1}$.

\begin{figure}[h]
\begin{center}
\includegraphics[width=7cm]{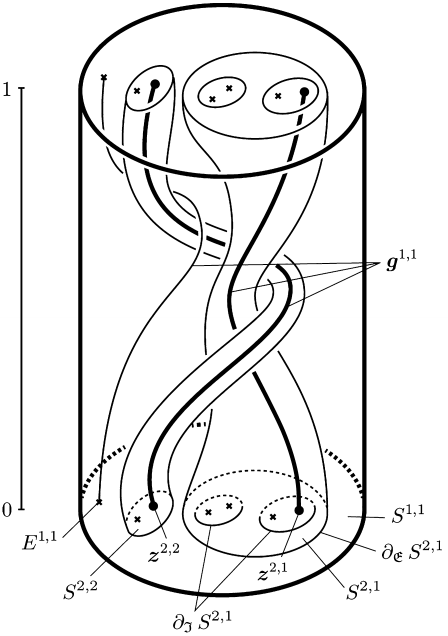}
\end{center}
\caption{The irreducible braid component $\widehat{\mathbold{b}({1,1})}$ of a pure braid $b$.}\label{fig6.1}
\end{figure}

\bigskip

\noindent {\bf The irreducible braid component $\widehat{\mathbold{b}(\ell,j)}$ associated
to $S^{\ell,j}$ with $\ell>1$.} Take a component $S^{\ell,j}$ of generation
$\ell$. Recall that $E^{\ell,j} = E \cap S^{\ell,j}$. \index{$E^{\ell,j}$}Consider
the interior boundary components of $S^{\ell,j}$. Each such boundary
component is the exterior boundary of a component $S^{\ell +1,j'}$
of ${\mathbb P}^1 \backslash \underset{C \in \mathcal{C}}{\bigcup}
C_i$ of generation $\ell +1$. Let $\delta^{\ell +1,j'} \subset
{\mathbb C}$ be the topological disc bounded by $\partial_{\mathfrak
E} \, S^{\ell +1,j'}$, and let $\mathbold{z}^{\ell +1,j'}$ be a point in $E
\cap \delta^{\ell +1,j'}$. Let ${\mathbold E}^{\ell,j}$ be the
\index{$\mathbold{E}^{\ell,j}$} collection of $\mathbold{z}^{\ell +1,j'}$
obtained in this way (one such point is assigned to each interior
boundary component of $S^{\ell ,j}$).

We obtain the tubular geometric braid
\begin{equation}
\label{eq6.5a} \varphi_t \left(E^{\ell,j} \cup
\bigcup_{j=1}^{k_{\ell+1}} \, \overline{\delta^{\ell+1,j'}} \right) \, ,
\quad t \in [0,1] \, ,
\end{equation}
and the geometric fat braid
\begin{equation}
\label{eq6.5} \varphi_t \left(E^{\ell,j} \cup {\mathbold{E}}^{\ell,j}\right) \, , \quad t \in [0,1] \, .
\end{equation}
The geometric fat braid is obtained by taking a deformation retraction of each of the cylinders
$\varphi_t \left(\overline{\delta^{\ell+1,j'}}\right)$ to the fat strand of
$b$ contained in this cylinder.

Let $n(\ell,j)$ be \index{$n(\ell,j)$} the number of points of
$E^{\ell,j} \cup {\mathbold{E}}^{\ell,j}$. We will write the
geometric fat braid (\ref{eq6.5}) as a mapping $ \mathbold{g}^{\ell,j} : [0,1]
\to C_{n(\ell,j)} ({\mathbb C})\diagup {\mathcal{S}_{n(\ell,j)}}$ from \index{$\mathbold{g}^{\ell,j}$} the unit
interval into $C_{n(\ell,j)} ({\mathbb C})\diagup {\mathcal{S}_{n(\ell,j)}}$.

Denote the isotopy class of the geometric fat braid by $\mathbold{b}(\ell,j)$.
Intuitively, the fat braid $\mathbold{b}(\ell,j)$ is obtained from the braid $b$ by
discarding all strands with initial point not in $E^{\ell,j} \cup
{\mathbold{E}}^{\ell,j}$ and indicating the set of fat strands.

The conjugacy class $\widehat{\mathbold{b}(\ell,j)}$ of the fat braid  $\mathbold{b}(\ell,j)$
is the irreducible component of the braid $b$ associated to $S^{\ell,j}$. It depends only on the conjugacy class $\hat b$ of the braid and the isotopy class of the curve system $\mathcal{C}$.

The decomposition of the conjugacy class of a braid into irreducible braid components is described.

\medskip

\section[Recovery of conjugacy classes of pure braids.]{Recovery of the conjugacy class of a pure braid from the irreducible braid components}\label{sec:6.3}

\noindent {\bf Recovery of the geometric braid $g$ from the $\mathbold{g}^{\ell,j}$.}
Before we recover the conjugacy class $\hat b$
from its irreducible braid components (with respect to an admissible system of curves that completely reduces $\mathfrak{m}_{b,\infty}$), we write
the representing geometric braid $g:[0,1]\to C_n(\mathbb{C})\diagup \mathcal{S}_n$, that is obtained from the parameterizing isotopy (see \eqref{eq6.2a}),  in terms of the
geometric fat braids $\mathbold{g}^{\ell,j}$ constructed in the previous section.

We will use the following operation on points of symmetrized configuration spaces. Consider a point $E\cup \mathbold{E}\in C_{n}(\mathbb{C})\diagup \mathcal{S}_n$ (which is also considered as
a subset of $\mathbb{C}$).
The points in the set $\mathbold{E}$ are declared to be fat points. Let $\mathbold{z}'$ be one of the fat points. Take a subset $E'$ of $\mathbb{C}$ that contains $n'$ points (also
considered as point in  $C_{n'}(\mathbb{C})\diagup \mathcal{S}_{n'}  $. ) $E'$ may contain fat points or not. Suppose that $E'$ does not intersect $E\cup (\mathbold{E} \setminus \{\mathbold{z}'\})$.
Remove the point $\mathbold{z}'$ from the set $E\cup \mathbold{E}$ and add the points of the set $E'$. The resulting set is denoted by
\begin{equation}\label{eq6.6}
(E\cup \mathbold{E})\sqcup_{ \mathbold{z}'}E' \stackrel{def}= (E\cup \mathbold{E}' \setminus \{\mathbold{z}'\}) \cup E'.
\end{equation}
The set \eqref{eq6.6} may be considered as subset of $\mathbb{C}$ consisting of
 $n+n'-1$ points,
or as point in $C_{n+n'-1}(\mathbb{C})\diagup \mathcal{S}_{n+n'-1} $.
\index{$(E\cup \mathbold{E})\sqcup_{ \mathbold{z}'}E'$}

The respective operation on mappings into symmetrized configuration spaces is defined pointwise. More detailed, let $\mathbold{g}(t), t\in [0,1],$ be a geometric fat braid, let $\mathbold{g}_1(t), t\in [0,1],$ be a fat strand of $\mathbold{g}$, and let $\mathfrak{T}= \cup_{t\in[0,1]}\mathfrak{T}(t)$ be a tube that contains the fat strand  $\mathbold{g}_1$ (i.e. $\mathbold{g}_1(t)\in \mathfrak{T}(t), t\in [0,1],  $)
and does not intersect any other strand (i.e. $\mathfrak{T}(t)$ does not intersect the set
$\mathbold{g}(t)\setminus \mathbold{g}_1(t),$
$t\in [0,1] $). Suppose $\mathbold{g}'(t), t\in[0,1],$ is another geometric fat braid, such that $\mathbold{g}'$ is contained in the tube $\mathfrak{T}$, i.e.  $\mathbold{g}'(t)\in \mathfrak{T}(t),t\in[0,1] $. By removing the fat strand ${\mathbold{g}_1}$ and inserting the geometric fat braid  $\mathbold{g}'$ which is contained in the tube $\mathfrak{T}$, we obtain a new geometric fat braid  $\mathbold{g}(t)\sqcup_{\mathbold{g}_{1}(t)}  \mathbold{g}'(t), t\in[0,1],$ that  is denoted by $\mathbold{g}\sqcup_{\mathbold{g}_{1}}  \mathbold{g}'$,
\index{$\mathbold{g}\sqcup_{\mathbold{g}_{1}}\mathbold{g}'$}

Consider now the geometric fat braid
$\mathbold{g}^{1,1}: [0,1] \to \mathcal{C}_{n(1,1)}(\mathbb{C})\diagup {\mathcal{S}_{n(1,1)}}$ that was constructed in the previous section. Its fat strands have initial point in $\mathbold{E}^{1,1}$. Denote the points in $\mathbold{E}^{1,1}$ by $\mathbold{z}^{2,1},\ldots,\mathbold{z}^{2,k_2}$. Let the fat strand with initial point $\mathbold{z}^{2,j}$ be $\mathbold{g}_{{2,j}}$.

Consider the geometric fat braids of second generation. Start with $\mathbold{g}^{2,1}:[0,1] \to \mathcal{C}_{n(2,1)}\diagup{\mathcal{S}_{n(2,1)} } $ with fat strands having initial point in $\mathbold{E}^{2,1}$ (see Section \ref{sec:6.2}).
Consider the geometric fat braid
\begin{equation}\label{eq6.12}
\mathbold{g}^{1,1} \sqcup_{\mathbold{g}_{2,1}}  \mathbold{g}^{2,1} : [0,1] \to C_{n(1,1)+n(2,1)-1}
({\mathbb C})\diagup{\mathcal{S}_{n(1,1)+n(2,1)-1} }\, ,
\end{equation}
obtained by replacing the strand $ \mathbold{g}_{2,1}$ of $ \mathbold{g}^{1,1}$ by the geometric fat braid $ \mathbold{g}^{2,1}$. The mapping \eqref{eq6.12} defines a geometric fat braid whose fat strands have initial point in $(\mathbold{E}^{1,1} \setminus\{\mathbold{z}^{2,1}\} \cup  \mathbold{E}^{2,1})$. The base point of \eqref{eq6.12} is
\begin{align}
E^{1,1} \cup (\mathbold{ E}^{1,1} \backslash
\{ \mathbold{z}^{2,1} \} ) \cup (E^{2,1} \cup \mathbold{E}^{2,1})
= (E^{1,1} \cup \mathbold{ E}^{1,1})\sqcup_{\mathbold{z}^{2,1}}(E^{2,1} \cup \mathbold{E}^{2,1})\,. \nonumber
\end{align}
We may identify
\eqref{eq6.12} with the geometric fat braid
\begin{equation}
\label{eq6.11} \varphi_t \left( E^{1,1}  \cup (\mathbold{E}^{1,1} \backslash
\{\mathbold z^{2,1} \} ) \cup (E^{2,1} \cup \mathbold{E}^{2,1}) \right) \, ,
\quad t \in [0,1] \, .
\end{equation}
This geometric fat braid is obtained by forgetting all strands of $g$ with initial point not in $ E^{1,1} \cup (\mathbold{E}^{1,1} \backslash
\{{z}^{2,1} \} ) \cup E^{2,1}  \cup \mathbold{E}^{2,1}$ and indicating the set of fat strands.

In the same way we define by induction on $j=1,\ldots, k_2$ the geometric fat braid
\begin{equation}\label{eq6.11a'}
\mathbold{g}^{1,1} \sqcup_{\mathbold{g}_{2,1}} \mathbold{g}^{2,1}\sqcup \ldots \sqcup_{\mathbold{g}_{2,k_2}}\mathbold{g}^{2,k_2}.
\end{equation}
The geometric fat braid \eqref{eq6.11a'} is associated to the set
\begin{equation}\label{eq6.11a''}
Q_2=S^{1,1} \cup S^{2,1} \cup {C}^{2,1}\cup\ldots \cup S^{2,k_2}\cup {C}^{2,k_2}
\end{equation}
(see also Section \ref{sec:6.1}). In other words, each point in $E \cap Q_2$ is the initial point of an ordinary strand, and each hole of $Q_2$ contains a point in $E$ which is the initial point of a fat strand of one of the $\mathbold{g}^{3,j}$.
Since $\mathbold{E}^{1,1} \setminus
\{\mathbold{z}^{2,1},\ldots,\mathbold{z}^{2,k_2}\}$ is the empty set,
the set of all initial points of fat strands of the
geometric fat braid \eqref{eq6.11a'}
is equal to $\bigcup _{j=1}^{k_2} \mathbold{E}^{2,j}$. The geometric fat braid \eqref{eq6.11a'} can be identified with the geometric fat braid defined by
\begin{align}\label{eq6.11a}
\varphi_t \left(E^{1,1}  \cup (E^{2,1} \cup \mathbold{E}^{2,1})
\cup \ldots \cup (E^{2,k_2} \cup \mathbold{E}^{2,k_2}) \right),\; t \in [0,1].
\end{align}
The set of all initial points of this geometric fat braid equals
\begin{align}
E^{1,1}\cup &(E^{2,1}\cup  \mathbold{E}^{2,1})\cup \ldots\cup ({E}^{2,k_2}\cup \mathbold{E}^{2,k_2})\nonumber\\
=E^{1,1}\cup &(E^{2,1}\ldots\cup E^{2,k_2})\cup (\mathbold{E}^{2,1}\cup\ldots\cup\mathbold{E}^{2,k_2})\,.\nonumber
\end{align}

Define by induction on $\ell\geq 3$ the geometric fat braid associated to $Q_{\ell}$ (see equality  \eqref{eq6.0} in Section \ref{sec:6.1}) using the fat strands of the geometric fat braid associated to $Q_{\ell-1}$. The fat strands of the geometric fat braid associated to $Q_{\ell-1}$ are those strands of $g$ that have its initial point in  $\mathbold{E}^{\ell-1,1}\cup \ldots\cup \mathbold{E}^{\ell-1,k_{\ell-1}}$. We replace each point $\mathbold{z}^{\ell,j}$ in the latter set by the respective set ${E}^{\ell,j}\cup\mathbold{E}^{\ell,j}$.
By induction the base point of the geometric braid associated to $\varphi_t$ and $Q_{\ell}$ equals
\begin{equation}\nonumber
E^{1,1}\cup (E^{2,1}\cup \ldots\cup E^{2,k_2})\cup\ldots\cup (E^{\ell,1}\cup \ldots\cup E^{\ell,k_{\ell}})\cup(\mathbold{E}^{\ell,1}\cup\ldots\cup \mathbold{E}^{\ell,k_{\ell}})\,.
\end{equation}

The innermost geometric fat braids $\mathbold{g}^{N,j}$ have only ordinary strands, hence they are merely geometric braids. Therefore, the geometric fat braid
\begin{equation}\label{eq12a}
\mathbold{g}^{1,1} \sqcup_{\mathbold{g}_{2,1}}  \mathbold{g}^{2,1}\sqcup \ldots \sqcup_{\mathbold{g}_{2,k_2}} \mathbold{g}^{2,k_2}\sqcup \ldots \sqcup_{\mathbold{g}_{N,1}} \mathbold{g}^{N,1}\sqcup \ldots \sqcup _{\mathbold{g}_{N,k_N}} \mathbold{g}^{N,k_N}
\end{equation}
is an ordinary geometric braid.
By the induction its base point equals $E$. Hence,
the braid
that is defined by \eqref{eq12a} for $\ell=N$
coincides with $g$. We recovered the particular geometric braid $g$ representing $b$ (see \eqref{eq6.1a}) from the particular geometric fat braids given by $\mathbold {g} ^{\ell,j}$.

\medskip

\noindent {\bf Recovery of the conjugacy class of a pure braid from its irreducible braid components.}
We will show now how to recover the conjugacy class of a pure braid from knowing its irreducible braid components (with respect to the isotopy class of a chosen admissible system of curves that completely reduces the class $\mathfrak{m}_{b,\infty}$ of the braid). By induction on the generation $\ell$
we will choose suitable geometric braids representing the irreducible braid components $\widehat{\mathbold{b}(\ell,j)}$ and put them together in
the same way as the $\mathbold{g}^{\ell,j}$ were put together. It needs to be done in such a way that at each step we obtain a geometric fat braid that is isotopic to the fat braid obtained by putting together the respective mappings $\mathbold{g}^{\ell,j}$.

We will use the following operations on points in $C_n(\mathbb{C})\diagup \mathcal{S}_n$ and on mappings from an interval to $C_n(\mathbb{C})\diagup \mathcal{S}_n$.
For an element $E \in C_n(\mathbb{C})\diagup \mathcal{S}_n$ and a number $\varepsilon$ we denote by $\varepsilon E$ the element $E \in C_n(\mathbb{C})\diagup \mathcal{S}_n$
which is obtained by multiplying each point in $E$ (considered as unordered set of points in $\mathbb{C}$) by $\varepsilon$. For a point $z \in \mathbb{C}$ we let
\begin{equation} \label{eq6.7'}
z \boxplus \varepsilon E
\end{equation}
be the set that is obtained from $\varepsilon E$ by adding $z$ to each point in $\varepsilon E$.\index{$z \boxplus \varepsilon E$}

Similarly, for a function $f_1:[0,1]\to \mathbb{C}$, a number $\varepsilon$, and a mapping $f^2:[0,1] \to C_n(\mathbb{C})\diagup \mathcal{S}_n$ we denote by
\begin{equation}
\label{eq6.8'}  f_1 \boxplus {\varepsilon} f^2
\end{equation}
the \index{$f_1\boxplus \varepsilon f^2$} mapping
from $[0,1]$ to symmetrized configuration space defined for each $t$
by $f_1(t) \boxplus \varepsilon f^2(t)$.

We need the following two lemmas on isotopies of geometric fat braids that are constructed in this way.

\begin{lemm}\label{lem6.1'} Let $\mathbold{g}^0$ and 
$\mathbold{g}^1$ be two geometric fat braids contained in a closed tube $\mathfrak{T}$ (i.e. for each $t \in [0,1]$ the points $\mathbold{g}^0(t)$ and $\mathbold{g}^1(t)$
are contained in $\mathfrak{T} \cap (\{t\} \times \mathbb{C})$). If $\mathbold{g}^0$ and $\mathbold{g}^1$ are isotopic then there exists an isotopy of geometric fat braids which is contained in $\mathfrak{T}$ and joins the two braids.
\end{lemm}

\noindent {\bf Proof.}
Let $\mathbold{g}^s,\, s \in [0,1]$, be the smooth family
of geometric fat braids defined by the isotopy. Let $R>0$ be a number such that all geometric braids $\mathbold{g}^s, \, s \in [0,1],\,$ together with the tube $\mathfrak{T}$, are contained in $[0,1] \times R {\mathbb{D}}$. Let $\mathfrak{T}^0 \subset \mathfrak{T}$ be a tube
which contains both geometric braids $\mathbold{g}^0$ and $\mathbold{g}^1$, such that the fiber $\mathfrak{T}^0 \cap (\{t\} \times \mathbb{C}) $ is relatively compact in the fiber  $\mathfrak{T} \cap (\{t\} \times \mathbb{C} )$ for each $t\in [0,1]$.

Let $\chi$ be a self-homeomorphism of $[0,1] \times {\mathbb{C}}$ which preserves each fiber $\{t\} \times   {\mathbb{D}}$, is equal to the identity  on $\mathfrak{T}^0 $
and outside $[0,1] \times  R_1 \overline{\mathbb{D}}$ for a large number $R_1>R$,  and maps the set $[0,1] \times  R \overline{\mathbb{D}}$ onto $\mathfrak{T}$.

Then the equalities $\chi \circ \mathbold{g}^0 = \mathbold{g}^0 $ and $\chi \circ \mathbold{g}^1 =\mathbold{ g}^1$ hold, and $\chi \circ \mathbold{g}^s, \, s \in [0,1],\,$ is a smooth family
of geometric fat braids contained in $\mathfrak{T}$ that joins $\mathbold{g}^0$ with
$\mathbold{g}^1$. The lemma is proved.
\hfill $\Box$

\begin{lemm}\label{lem6.1''} Let $\mathbold{f}^1\,$, $\mathbold{f}^2\,$, $\mathbold{g}^1$, and $\mathbold{g}^2$ be pure geometric fat braids, such that $\mathbold{f}^1$ is free isotopic to $\mathbold{f}^2$, and $\mathbold{g}^1$ is free isotopic to $\mathbold{g}^2$. For $j=0,1$ we let $\mathbold{f}_1^j$ be a fat strand  of $\mathbold{f}^j$,  and let $\mathfrak{T}(\mathbold{f}_1^j)$ be a tube around $\mathbold{f}_1^j$ which does not meet the other strands of $\mathbold{f}^j\,$ and contains $\mathbold{g}^j$.
Then  $\mathbold{f}^1 \sqcup_{\mathbold{f}_1^1} \mathbold{g}^1$ and $\mathbold{f}^2 \sqcup_{\mathbold{f}_1^2}\mathbold{ g}^2$ are free isotopic geometric fat braids.
\end{lemm}

\noindent {\bf Proof.}
Note first that for any sufficiently small number $\varepsilon$ the braids defined by the mappings $\mathbold{g}^1$ and $\mathbold{f}_1^1 \boxplus \varepsilon \mathbold{g}^1$ are free isotopic.
Further, for small $\varepsilon$ the fat braid   $\mathbold{f}_1^1 \boxplus \varepsilon \mathbold{g}^1$ is contained in the tube $\mathfrak{T}(\mathbold{f}_1^1)$ around the fat strand $\mathbold{f}_1^1$ of $\mathbold{f}^1$.
Since,
$\mathbold{g}^1$ is also contained in $\mathfrak{T}(\mathbold{f}_1^1)$,
Lemma \ref{lem6.1'} (which applies to fat braids as well) provides an isotopy that is contained in $\mathfrak{T}(\mathbold{f}_1^1)$ and
joins $\mathbold{g}^1$
and  $\mathbold{f}_1^1 \boxplus \varepsilon \mathbold{g}^1$. As a consequence, there is a free isotopy joining $\mathbold{f}^1\sqcup_{\mathbold{f}_1^1}\mathbold{g}^1$ and
$\mathbold{f}^1\sqcup_{\mathbold{f}_1^1}({\mathbold{f}_1^1} \boxplus \varepsilon \mathbold{g}^1)$. In the same way for
small $\varepsilon$ the braids $\mathbold{f}^2\sqcup_{\mathbold{f}_1^2} \mathbold{g}^2$ and
$\mathbold{f}^2\sqcup_{\mathbold{f}_1^2}(\mathbold{f}_1^2 \boxplus \varepsilon \mathbold{g}^2)$ are free isotopic. Hence,
for any a priory given small positive number $\varepsilon$ we may assume from the beginning that $\mathbold{g}^1= \mathbold{f}_1^1\boxplus \varepsilon \mathbold{\tilde g}^1$, and $\mathbold{g}^2=\mathbold{f}_1^2\boxplus \varepsilon \mathbold{\tilde g}^2$, where $\mathbold{\tilde g}^1$, and $\mathbold{\tilde g}^2$ are geometric fat braids that are contained in the cylinder $[0,1]\times \mathbb{D}$.

Let $\mathbold{f}^s,\, s \in [1,2],$ be an isotopy of braids joining $\mathbold{f}^1$ with $\mathbold{f}^2$,
and let $\mathbold{f}_1^s$ be the family of strands joining $\mathbold{f}_1^1$ with $\mathbold{f}_1^2$.
Also, let $\mathbold{\tilde g}^s,\,s \in [1,2],$ be an isotopy of braids contained in the cylinder $[0,1]\times \mathbb{D}$, that join $\mathbold{\tilde g}^1$ and $\mathbold{\tilde g}^2$.
Choose $\varepsilon>0$ so that for each $s\in [1,2]$ the $\varepsilon$-tubes around each strand of  $\mathbold{f}^s$ are disjoint. Then $\mathbold{f}^s\sqcup_{\mathbold{f}_1^s}(\mathbold{f}_1^s \boxplus \varepsilon \mathbold{\tilde g}^s),\, s \in [1,2]$, is an isotopy of braids that joins $\mathbold{f}^1\sqcup_{\mathbold{f}_1^1}(\mathbold{f}_1^1 \boxplus \varepsilon \mathbold{\tilde g}^1)$ with $\mathbold{f}^2\sqcup_{\mathbold{f}_1^2}(\mathbold{f}_1^2 \boxplus \varepsilon \mathbold{\tilde g}^2)$.
The Lemma is proved. \hfill $\Box$

\bigskip

We describe now the essential part of the recovery procedure.
Suppose we know the irreducible braid components of a pure braid $b$, i.e. we know the conjugacy classes
of the fat braids $\widehat{\mathbold{b}(\ell,j)}$
\index{$\widehat{\mathbold{b}(\ell,j)}$} for all $\ell$
and $j$. We want to recover the conjugacy class $\hat b$.

Represent each conjugacy class of fat braids $\widehat{\mathbold{b}(\ell,j)}$ by a geometric fat braid
$$
\mathbold{f}^{\ell,j}:[0,1] \to C_{n(\ell,j)}(\mathbb{C})
\diagup \mathcal{S}_{n(\ell,j)}
$$
which is contained in the tube $[0,1]\times \mathbb{D}$ and has base point $\tilde{E}^{\ell,j}\cup \tilde{\mathbold{E}}^{\ell,j}$ .
The set of initial points of the fat strands equals $\tilde{\mathbold{E}}^{\ell,j}$ and $n(\ell,j)$ is the total number of strands  (fat strands and ordinary strands together) of $\mathbold{f}^{\ell,j}$.

Let $\tilde{\mathbold{E}}^{1,1}=\{\tilde{\mathbold{z}}^{2,1},\ldots, \tilde{\mathbold{z}}^{2,k_2} \}$.
Denote by $ \mathbold{f}_{2,j}$ the strand of $\mathbold{f}^{1,1}$ \index{$\mathbold{f}_{2,j}$}
with initial point
$\mathbold{ f}_{2,j} (1) = \tilde{\mathbold{z}}^{2,j}$.
Take a small positive number $\varepsilon_2$ such that the $\varepsilon_2$-neighbourhoods of the strands of $\mathbold{f}^{1,1}$ are pairwise disjoint.

We consider the pure geometric fat braid
\begin{equation}
\label{eq6.9}\mathbold{ f}^{1,1} \sqcup_{\mathbold{f}_{2,1}}  ( \mathbold{f}_{2,1} \boxplus {\varepsilon}_2 \,
\mathbold{f}^{2,1}) : [0,1] \to C_{n(1,1)+n(2,1)-1} ({\mathbb C})\diagup \mathcal{S}_{n(1,1)+n(2,1)-1} \, .
\end{equation}
\index{$\mathbold{f}^{1,1} \sqcup_{\mathbold{f}_{2,1}}  ( \mathbold{f}_{2,1} \boxplus {\varepsilon}_2 \,
\mathbold{f}^{2,1})$}
Its base point
is the point
\begin{equation}
\label{eq6.9a} \tilde{\mathbold{E}}^{1,1} \sqcup_{\tilde{\mathbold{z}}^{2,1}} (\tilde{\mathbold{z}}^{2,1} \boxplus \varepsilon_2 \,
(\tilde{E}^{2,1} \cup \tilde{\mathbold{E}}^{2,1})) \,.
\end{equation}
Since the geometric fat braids $\mathbold{f}^{1,1}$ and  $\mathbold{g}^{1,1}$ (see \eqref{eq6.1} and \eqref{eq6.3}) both represent $\widehat{\mathbold{b}(1,1)}$, and the geometric fat braids $\mathbold{g}^{2,1}$ (see \eqref{eq6.5}) and $\mathbold{f}^{2,1}$ both represent $\widehat{\mathbold{b}(2,1)}$, Lemma \ref{lem6.1''} shows that for small $\varepsilon_2$ the geometric fat braids \eqref{eq6.9} and \eqref{eq6.12}  are free isotopic.

The free isotopy class of braids corresponding to \eqref{eq6.9}, or
equivalently to \eqref{eq6.12}, is denoted by $\widehat {\mathbold{b}(1,1)}
\sqcup \widehat {\mathbold{b}(2,1)}.\,$

\medskip

We proceed now in the same way by induction on $j=2,\ldots, k_2,\,$. 
We arrive at a geometric fat braid
\begin{equation}
\label{eq6.17a}  \mathbold{f}^{1,1} \sqcup_{\mathbold{f}_{2,1}} ( \mathbold{f}_{2,1} \boxplus \varepsilon_2
\mathbold{f}^{2,1}) \sqcup \ldots \sqcup_ {\mathbold{f}_{2,k_2}} ( \mathbold{f}_{2,k_2} \boxplus \varepsilon_2
\mathbold{f}^{2,k_2})\,,
\end{equation}
which is, by repeated application of Lemma \ref{lem6.1''}, free isotopic to
the geometric fat braid defined by \eqref{eq6.11a'} and \eqref{eq6.11a}.

We make a similar induction for $\ell=2,\ldots, N$.
Suppose we obtained the geometric fat braid
\begin{align}\label{eq6.18}
\mathbold{f}_{Q_{\ell}}\stackrel{def}=\mathbold{f}_1^{1,1} \sqcup_{\mathbold{f}_{2,1}}(\mathbold{f}_{2,1} \boxplus \varepsilon_2
\mathbold{f}^{2,1}) \sqcup  \ldots & \sqcup_{\mathbold{f}_{2,k_2}}( \mathbold{f}_{2,k_2} \boxplus \varepsilon_2
 \mathbold{f}^{2,k_2}) \nonumber\\
\sqcup  \ldots & \sqcup_{\mathbold{f}_{\ell,k_{\ell}}}( \mathbold{ f}_{\ell,k_{\ell}} \boxplus
\varepsilon_{\ell} \mathbold{ f}^{\ell,k_{\ell}})
\end{align}
that is free homotopic to the geometric fat braid \eqref{eq12a} with $N$ replaced by $\ell$. Denote the fat strands of \eqref{eq6.18} by $\mathbold{f}_{\ell+1,j}$.
Choose a  positive number $\varepsilon_{\ell+1}$ such that the $\varepsilon_{\ell+1}$-neighbourhoods of the strands of the geometric fat braid \eqref{eq6.18} are pairwise disjoint. Add strands to \eqref{eq6.18} by the operation
\begin{align}
\mathbold{f}_{Q_{\ell}}\sqcup_{\mathbold{f}_{\ell+1,1}}( \mathbold{f}_{\ell+1,1} \boxplus
\varepsilon_{\ell+1} \mathbold{f}^{\ell+1,1})\sqcup\ldots
\sqcup_{\mathbold{f}_{\ell+1,k_{\ell+1}}}( \mathbold{f}_{\ell+1,k_{\ell+1}} \boxplus
\varepsilon_{\ell+1} \mathbold{ f}^{\ell+1,k_{\ell+1}})\,.
\end{align}

We obtain a geometric fat braid of the same form as \eqref{eq6.18} with $\ell$ replaced by $\ell+1$.
The geometric fat braid $\mathbold{f}^{N,k_{N}}$ is an ordinary geometric braid. Hence, the induction terminates for $\ell=N$ at an ordinary geometric braid
\begin{align}\label{eq6.18''}
\mathbold{f}_1^{1,1}  \sqcup_{\mathbold{f}_{2,1}}( \mathbold{f}_{2,1} \boxplus \varepsilon_2
\mathbold{f}^{2,1}) \sqcup \ldots & \sqcup_{\mathbold{f}_{2,k_2}} ( \mathbold{f}_{2,k_2} \boxplus \varepsilon_2
\mathbold{f}^{2,k_2}) \nonumber\\
\sqcup \ldots & \sqcup_{\mathbold{f}_{N,k_N}} ( \mathbold{ f}_{N,k_N} \boxplus
\varepsilon_N \mathbold{f}^{N,k_N})
\end{align}
for successively chosen small numbers $\varepsilon_{\ell}$.

We denote the class represented by \eqref{eq6.18''} by
\begin{equation}
\label{eq6.18c} \widehat{\mathbold{b}(1,1)} \sqcup \widehat{\mathbold{b}(2,1)} \sqcup \ldots
\sqcup \widehat{\mathbold{b}(2,k_2)} \sqcup \ldots \sqcup \widehat{\mathbold{b}(N,k_N)} .
\end{equation}
Applying Lemma \ref{lem6.1''} at each step of the construction, we see that the class \eqref{eq6.18c} equals $\hat b$.

We showed how to recover the conjugacy class of reducible pure braids knowing their irreducible braid components.

\bigskip

\noindent {\bf The relation between irreducible nodal components and irreducible braid components.}

Recall that for each $(\ell,j)$ we denoted by $\widehat{{\mathfrak m}_{b,\odot}^{\ell,j}}$  the irreducible nodal component that is associated to ${\mathfrak m}_{b,\infty}$ and the isotopy class of the admissible system of curves $\mathcal{C}$. For the irreducible braid component $\widehat{\mathbold{b}({\ell,j})}$, that is a conjugacy class of fat braids which was associated to $b$ and the isotopy class of $\mathcal{C}$, we now consider the respective conjugacy class of ordinary braids $\widehat{{b}({\ell,j})}$ by declaring all strands to be ordinary. Let $\reallywidehat{{\mathfrak m}_{b(\ell,j),\infty}}$ be  the  conjugacy class of mapping classes corresponding to the braid class $\reallywidehat{{b}({\ell,j})}$. The following lemma holds.

\begin{lemm}\label{lem6.1}
For each $(\ell,j)$ the equality
\begin{equation}\label{eq6.27}
\widehat{{\mathfrak m}_{b,\odot}^{\ell,j}}=\reallywidehat{{\mathfrak m}_{b(\ell,j),\infty}}
\end{equation}
is satisfied.
\end{lemm}
In the statement of the lemma either both conjugacy classes are considered as classes of self-homeomorphisms of a punctured Riemann surface or both are considered as classes of self-homeomorphisms of a closed Riemann surface with distinguished points.

\smallskip

\noindent {\bf Proof of Lemma \ref{lem6.1}. }
Recall that $E'_n=E_n\cup \{\infty\}$, and
the mapping $\varphi_{b,\infty}|{S^{\ell,j}}\setminus E'_n$ represents the class $\widehat{{\mathfrak m}_{b,\odot}^{\ell,j}}$ considered as element of
$\widehat{\mathfrak{M}}(Y^{\ell,j})$. As before we will identify the conjugacy class
$\widehat{\mathfrak{M}}(Y^{\ell,j}) $ on the punctured surface with the conjugacy class
$\widehat{\mathfrak{M}}((Y^{\ell,j})^c;(\mathcal{N}\cap (Y^{\ell,j})^c) \cup w^{\ell,j}( { E'}^{\ell,j}))$ on the closed surface.

The irreducible braid component $\widehat{b(\ell,j)}$ is represented by the geometric braid \eqref{eq6.5} after forgetting that some strands are declared to be fat.
The parameterizing family  $\varphi_t$ that was chosen for the  geometric braid \eqref{eq6.1a} representing the braid $b$ 
serves also as  parameterizing family for the latter geometric braid.

The  mapping class $\mathfrak{m}_{b(\ell,j),\infty}\in  \mathfrak{M}(\mathbb{P}^1;  {E'}^{\ell,j} \cup \mathbold{E}^{\ell,j})$ that is associated to the braid $b(\ell,j)$ (see \eqref{eq6.5})
is the class of $\varphi_1=\varphi_{b,\infty}$ in $\mathfrak{M}(\mathbb{P}^1;  {E'}^{\ell,j} \cup \mathbold{E}^{\ell,j})$ identified with  $\mathfrak{M}(\mathbb{P}^1\setminus ( {E'}^{\ell,j} \cup \mathbold{E}^{\ell,j}))$. This follows from the fact that
$\varphi_t$ is a parameterizing isotopy also for the geometric braid \eqref{eq6.5} that represents
$b(\ell,j)$.

The mapping $\varphi_{b,\infty}\mid (\mathbb{P}^1\setminus ({E'}^{\ell,j}\cup \mathbold{E}^{\ell,j}))$ is obtained from the mapping $\varphi_{b,\infty}\mid (S^{\ell,j}\setminus E'_n)$ by isotopy and conjugation. Indeed,
consider all holes of $S^{\ell,j}$. These are all discs 
$\delta^{\ell+1,i'}$ that are adjacent to an interior boundary component of $S^{\ell,i}$. Each disc $\delta^{\ell+1,i'}$ contains a point $\mathbold{z}^{\ell+1,i'}\in \mathbold{E}^{\ell,j}$ and each point in $\mathbold{E}^{\ell,j}$ is contained in one of the discs. If $(\ell,j)\neq (1,1)$ then there is one more hole of $S^{\ell,j}$, namely the complement of $\overline{\delta^{\ell,j}}$. This disc contains $\infty$ and no other point of $E'^{\ell,j}$. We denote it by $\delta_{\infty}^{\ell ,j}$.

For each hole $\delta^{\ell+1,i'}$ (and  $\delta_{\infty}^{\ell ,j}$, if $\ell \neq 1$)
we take an annulus $A^{\ell+1,i'}$ (and $A_{\infty}^{\ell ,j}$ respectively, if $\ell \neq 1$),
that is contained in $S^{\ell,j}$ and shares a boundary component with the boundary of the hole. The annuli are chosen to have disjoint closure. For each disc  ${\delta}^{\ell+1,i'}$ ( $\delta_{\infty}^{\ell ,j}$, respectively) we consider the disc ${\delta'}^{\ell+1,i'}=\overline{{\delta}^{\ell+1,i'}}\cup  A^{\ell+1,i'}$ (and ${\delta'}_{\infty}^{\ell,j}= \overline{{\delta}_{\infty}^{\ell,j}}\cup A_{\infty}^{\ell,j}$, respectively).

There is a homeomorphism $\tilde{ w}^{\ell,j}$ from ${S^{\ell,j}}\setminus {E'}^{\ell,j}$ onto $\mathbb{P}^1\setminus ( {E'}^{\ell,j} \cup \mathbold{E}^{\ell,j})$ with the following properties. $\tilde{ w}^{\ell,j}$  is equal to the identical injection on the complement 
of all annuli in $S^{\ell,j}$. 
Then $\tilde{ w}^{\ell,j}$ maps each annulus $A^{\ell+1,i'}$ ($A_{\infty}^{\ell,j}$, respectively) homeomorphically onto the respective punctured disc ${\delta'}^{\ell+1,i'}\setminus \{\mathbold{z}^{\ell+1,i'}\}$ (${\delta'}_{\infty}^{\ell,j}\setminus\{\infty\}$, respectively)
and its continuous extension to the boundary component  $\partial {\delta'}^{\ell+1,i'}$ of $A^{\ell+1,i'}$
(to the boundary component $\partial {\delta'}_{\infty}^{\ell,j}$ of $A_{\infty}^{\ell,j}$, respectively)
equals the identity on the boundary component.

The mapping $(\tilde{ w}^{\ell,j})^{-1}$ conjugates $\varphi_{b,\infty}\mid S^{\ell,j}\setminus {{E}'}^{\ell,j}$ to a self-homeomorphism  $\varphi'_{b,\infty} $  of $\mathbb{P}^1\setminus
({{E}'}^{\ell,j}\cup\mathbold{E}^{\ell,j}) $ that differs from  $\varphi_{b,\infty}\mid \mathbb{P}^1\setminus ({E'}^{\ell,j}\cup\mathbold{E}^{\ell,j}) $ only on the union of the punctured discs
${\delta'}^{\ell+1,i'}\setminus \{\mathbold{z}^{\ell+1,i'}\}$ (${\delta'}_{\infty}^{\ell,j}\setminus\{\infty\}$, respectively)
and maps each punctured disc onto itself. Moreover, $\varphi'_{b,\infty} \circ \varphi_{b,\infty}^{-1}$ is equal 
to the identity on the boundary of each disc. Since two self-homeomorphisms of a punctured disc that fix the boundary circle pointwise are isotopic to the identity through self-homeomorphisms of the punctured disc that fix the boundary circle pointwise, the conjugated homeomorphism is isotopic to $\varphi_{b,\infty}\mid \mathbb{P}^1\setminus {E'}^{\ell,j}$.
The equality \eqref{eq6.27} is proved. \hfill  $\Box$

\section{Pure Braids, the reducible case. Proof of the Main Theorem.}
\label{sec:6.4}

In this section we will prove the Main Theorem for reducible pure braids. The theorem will follow from the Propositions \ref{prop7.1} and \ref{prop7.2} below which are of independent interest.

We will represent a conjugacy class of pure braids $\hat b$ by a holomorphic map from an annulus $A$  to the symmetrized configuration space $C_n(\mathbb{C})\diagup \mathcal{S}_n$ (or by a holomorphic map on $A$ that extends continuously to the closure $\bar A$). The set $\{(z,f(z));\, z \in A\}$ (or $\{(z,f(z));\, z \in \bar A\}$, respectively) is the union of $n$ connected components  which we also call strands. A map with some strands declared to be fat will be called a fat map and denoted by $\mathbold{f}$.

\begin{prop}\label{prop7.1} The conformal module of a conjugacy class $\hat b$ of pure $n$-braids is equal to the smallest conformal module among the irreducible braid components of $\hat b$. In other words, the equality
\begin{equation}
\label{eq7.1} {\mathcal M} (\hat b) = \min_{\ell,j}
{\mathcal M} \left( \widehat{b(\ell,j)} \right)
\end{equation}
holds.
\end{prop}

\medskip

\noindent {\bf Proof.} Let $A$ be any annulus of conformal module $m(A)<{\mathcal M} (\hat b)$. Then there exists a holomorphic mapping
\begin{equation}\label{eq6.50}
f:  A  \to  C_n(\mathbb{C})\diagup \mathcal{S}_n
\end{equation}
representing $\hat b$.
Recall that for all $(\ell,j)$ the class $\widehat{b(\ell,j)}$ is obtained from $\hat b$ by forgetting some strands. Forgetting the respective strands of $f$ provides us a holomorphic map $f^{\ell,j}:A \to  C_{n(\ell,j)}(\mathbb{C})\diagup \mathcal{S}_{n(\ell,j)}$
that represents 
$\widehat{b(\ell,j)}$. Hence, for all $(\ell,j)$ the inequality
$\mathcal{M}(\widehat{b(\ell,j)}) > m(A)$ holds, i.e.
$$
{\mathcal M} (\hat b) \leq \min_{\ell,j}
{\mathcal M} \left( \widehat{b(\ell,j)} \right).
$$

\smallskip

To prove the opposite inequality, let $A$ be an annulus of conformal module $m(A)< \min_{\ell,j}
{\mathcal M} \left( \widehat{b(\ell,j)} \right)$. Then for each $(\ell,j)$ there is a continuous fat map $\mathbold{f}^{\ell,j}:\bar A \to C_{n(\ell,j)}(\mathbb{D})\diagup \mathcal{S}_{n(\ell,j)}$ that is holomorphic on $A$ and 
represents the conjugacy class of fat braids $\widehat{\mathbold{b}(\ell,j)}$.
We will denote by  $\mathbold{f}_{\ell+1,j'}$ the str ands of the mappings $\mathbold{f}^{\ell,j}$ that correspond to fat strands of $\widehat{\mathbold{b}(\ell,j)}$. 
As in the recovery procedure of the conjugacy class of braids from the irreducible braid components we
successively choose small numbers  $\varepsilon_{\ell}$ and consider the mapping
\begin{equation}\label{eq6.a*}
\mathbold{ f}^{1,1} \sqcup_{\mathbold{f}_{2,1}} (\mathbold{f}_{2,1} \boxplus \varepsilon_2
\mathbold{f}^{2,1}) \sqcup \ldots \sqcup_{\mathbold{f}_{2,k_2}} (\mathbold{f}_{2,k_2} \boxplus \varepsilon_2
 \mathbold{f}^{2,k_2}) \sqcup \ldots \sqcup_{\mathbold{f}_{N,k_N}} (  \mathbold{f}_{N,k_N} \boxplus
\varepsilon_N  \mathbold{f}^{N,k_N}) \,
\end{equation}
that are defined pointwise for $z \in \bar A$ in the same way as the respective objects were defined in section \ref{sec:6.3} for geometric fat braids.
The $\varepsilon_{\ell}$ are chosen successively so that the $\varepsilon_{\ell}$-neighbourhoods of the strands of the fat map \eqref{eq6.a*}
with $N$ replaced by $\ell-1$, $\ell\leq N,$ are pairwise disjoint.
Then the mapping \eqref{eq6.a*} is an ordinary mapping and represents the conjugacy class $\hat b$. Since all mappings $\mathbold{f}^{\ell,j}$ and $\mathbold{f}_{\ell',j'}$ are holomorphic on $A$, the mapping \eqref{eq6.a*} is holomorphic by construction.
Hence, $\mathcal{M}(\hat b) \geq m(A)$ for any $A$ with $m(A) <\min_{\ell,j}
{\mathcal M} \left( \widehat{b(\ell,j)} \right)$. We obtained the inequality
$$
\mathcal{M}(\hat b)\geq \min_{\ell,j}
{\mathcal M} \left( \widehat{b(\ell,j)} \right).
$$ The proposition is proved. \hfill $\Box$

\medskip

\begin{prop}\label{prop7.2} For the mapping class $\mathfrak{m}_{b,\infty}$ associated to a pure $n$-braid $b$ and its irreducible nodal components $\widehat{\mathfrak{m}_{b,\odot}^{\ell,j}}$ the equality
\begin{equation}
h(\widehat{\mathfrak{m}_{b,\infty}})= \max_{\ell,j}h(\widehat{\mathfrak{m}_{b,\odot}^{\ell,j}})
\end{equation}
holds.
\end{prop}
In the lemma all conjugacy classes are considered as classes of self-homeomorphisms of compact Riemann surfaces with distinguished points.

\medskip

\noindent {\bf Proof.}
We prove first the inequality
\begin{align}\label{eq6.51}
h (\widehat{{\mathfrak m}_{b,\infty}}) \geq \max_{\ell,j}h(\widehat{\mathfrak{m}_{b,\odot}^{\ell,j}}) .
\end{align}
We take again a self-homeomorphism $\varphi_{b,\infty}$ of $\mathbb{P}^1$ with distinguished points $ E'_n= E_n\cup \{\infty\}$ that fixes pointwise the complement of a large disc centered at $0$ and represents the mapping class $\mathfrak{m}_{b,\infty} \in \mathfrak{M}(\mathbb{P}^1; E'_n)$ of the braid $b$.
We require that $\varphi_{b,\infty}$  is completely
reduced by an admissible system of curves $\mathcal{C}$   that is contained in the unit disc $\mathbb{D}$. 
We choose a set $\mathbold{E}^{\ell,j}\subset E_n$ as follows. For each hole of $S^{\ell,j}$ bounded by an interior boundary component of $S^{\ell,j}$  we chose a point of $E$ contained in it. These points constitute the set  $\mathbold{E}^{\ell,j}$. Then
 $\mathfrak{m}_{b(\ell,j),\infty}$ is the mapping class of $\varphi_{b,\infty}$ in $\mathfrak{M}(\mathbb{P}^1;  {E'}^{\ell,j} \cup \mathbold{E}^{\ell,j})$. We obtain
\begin{align}\label{eq7.41a}
h({\mathfrak m}_{b(\ell,j),\infty}) = {\rm inf} \{ h(\varphi) : \varphi \;
{\rm is}
\;{\rm Hom}^+  (\mathbb{P}^1 ; {E'}^{\ell,j} \cup {\mathbold{E}}^{\ell,j})-\mbox{isotopic to $\varphi_{b,\infty}$ }\} \, .
\end{align}
On the other hand,
\begin{align}\label{eq7.40}
h({\mathfrak m}_{b,\infty}) = {\rm inf} \{ h(\varphi) : \varphi \;
{\rm is} \;
{\rm Hom}^+(\mathbb{P}^1 ; E'_n) - \mbox{isotopic to $\varphi_{b,\infty}$} \} \,.
\end{align}
Since the space which appears in (\ref{eq7.40}) is contained in the
space which appears in (\ref{eq7.41a}), the infimum in (\ref{eq7.40})
is not smaller than the infimum in (\ref{eq7.41a}):
\begin{equation}
\label{eq7.42} h({\mathfrak m}_{b,\infty}) \geq h ({\mathfrak m}_{b(\ell,j),\infty})
\quad \mbox{for each $\ell,j$.}
\end{equation}
Since by Lemma \ref{lem6.1} the equality  $ \widehat{{\mathfrak m}_{b,\odot}^{\ell,j}}= \reallywidehat{\mathfrak{m}_{b(\ell,j),\infty}}$ holds, the inequality \eqref{eq6.51} is proved.

\smallskip

The proof of the opposite inequality is based on Theorem  \ref{thm3.6}.
Recall that the class $\widehat{\mathfrak{m}_{b,\odot}^{\ell,j}}$ (considered as element of $\widehat{\mathfrak{M}}(Y^{\ell,j})$ for a punctured Riemann sphere  $Y^{\ell,j}$) is represented by the mapping $\varphi_{b,\infty}|({S^{\ell,j}} \setminus E'_n)$ on a Riemann surface of second kind. 
Take an absolutely extremal representative of this class. 
This is a self-homeomorphism of a punctured Riemann surface $\tilde {Y}^{\ell,j}$. It extends across the punctures to
a self-homeomorphism  $\tilde{\varphi}^{\ell,j}$ of $\mathbb{P}^1$ with distinguished points $\tilde{E'}^{\ell,j}\cup\tilde{\mathbold{E}}^{\ell,j} $. (The set $\tilde{E'}^{\ell,j}$ corresponds to $S^{\ell,j}\cap E_n'$, the set $\tilde{\mathbold{E}}^{\ell,j} $ corresponds to the holes of  $S^{\ell,j}$.) 
$\tilde{\varphi}^{\ell,j}$ is entropy minimizing in the class of self-homeomorphisms of $\mathbb{P}^1$ with distinguished points that is obtained from $\varphi_{b,\infty}|({S^{\ell,j}} \setminus E'_n)$ by isotopy, conjugation and extension across punctures, i.e.
$$
h(\tilde{\varphi}^{\ell,j}) =\inf\{h(\varphi): \varphi \in \widehat{\mathfrak{m}_{b,\odot}^{\ell,j}}\}\,.
$$
By Theorem 3.9 for each $(\ell,j)$ there is a closed topological disc around each distinguished point in  $\tilde{\mathbold{E}}^{\ell,j} $ 
and a self-homeomorphism $\tilde{\varphi}_0^{\ell,j}$ of
$\mathbb{P}^1$ with distinguished points $\tilde{E'}^{\ell,j}\cup \tilde{\mathbold{E}}^{\ell,j} $ 
that is isotopic (with distinguished points $\tilde{E'}^{\ell,j}\cup \tilde{\mathbold{E}}^{\ell,j} $) to $\tilde{\varphi}^{\ell,j}$, has the same entropy $h(\tilde{\varphi}_0^{\ell,j})= h(\tilde{\varphi}^{\ell,j})$, and equals the identity in a neighbourhood 
of each of the discs.

Since  $\tilde{\varphi}^{\ell,j} \in \widehat{\mathfrak{m}_{b,\odot}^{\ell,j}}$ there exists a homeomorphism  $\tilde {w}^{\ell,j}: S^{\ell,j} \setminus {E'}^{\ell,j} \to\tilde {Y}^{\ell,j} $ that conjugates $\tilde{\varphi}^{\ell,j}\mid \mathbb{P}^1\setminus ({E'}^{\ell,j}\cup \tilde{\mathbold{E}}^{\ell,j})$ to a self-homeomorphism of $S^{\ell,j} \setminus {E'}^{\ell,j}$ that is isotopic on this set to $\varphi_{b,\infty} \mid S^{\ell,j} \setminus {E'}^{\ell,j}$.
The conjugate
$(\tilde {w}^{\ell,j})^{-1}\circ \tilde{\varphi}_0^{\ell,j} \circ \tilde {w}^{\ell,j}$ of the isotopic homeomorphism $\tilde{\varphi}_0^{\ell,j}$
equals the identity in a neighbourhood of all boundary components of $S^{\ell,j}$. Hence, it extends to the boundary of $S^{\ell,j}$ as a self-homeomorphism $\varphi_0^{\ell,j}$ of $\overline{S^{\ell,j}}\setminus E'_n$ that fixes the boundary pointwise. The isotopy class in $\mathfrak{M}(\overline{S^{\ell,j}}\setminus E'_n; \partial S^{\ell,j})$ of the extension  $\varphi_0^{\ell,j}$ 
differs from that of the restriction $\varphi_{b,\infty} \mid \overline{S^{\ell,j}}\setminus E'_n$ 
by products of  powers of  Dehn twists about  curves in $S^{\ell,j}$ that are homologous to one of the boundary components of $S^{\ell,j}$.
By Lemma \ref{lemm3.11} for each $(\ell,j)$ the entropy minimizing  homeomorphism $\tilde{\varphi}_0^{\ell,j}$ may be chosen from the beginning so that ${\varphi}_0^{\ell,j}$
and $\varphi_{b,\infty} \mid \overline{S^{\ell,j}}\setminus E'_n$
are $\mathfrak{M}(\overline{S^{\ell,j}}\setminus E'_n; \partial S^{\ell,j})$-isotopic.

Consider the self-homeomorphism $\varphi_0$ of $\mathbb{P}^1\setminus E_n'$ whose restriction to  $\overline{S^{\ell,j}}\setminus E'_n$ equals $\varphi_0^{\ell,j}$ for all $\ell$ and $j$. (The self-homeomorphism is well defined since each $\varphi_0^{\ell,j}$ is equal to the identity near the boundary $\partial  S^{\ell,j}$.) The mapping  $\varphi_0$ is isotopic to $\varphi_{b,\infty}$ in $\mathfrak{M}(\mathbb{P}^1\setminus E'_n)$, hence it represents $\mathfrak{m}_{b,\infty}$ (considered as class of self-homeomorphisms of a punctured surface). For the 
extensions of the mappings 
across the punctures 
the inequality $h(\hat{\varphi}_0)\leq       \max_{\ell,j}h(\hat{\varphi}_0^{\ell,j})$ holds. 
Hence, the opposite inequality
\begin{align}\label{eq6.52}
h (\widehat{{\mathfrak m}_{b,\infty}}) \leq \max_{\ell,j}h(\widehat{\mathfrak{m}_{b,\odot}^{\ell,j}})
\end{align}
holds and the proposition is proved. \hfill $\Box$

\medskip

\noindent {\bf The Proof of the Main Theorem for reducible pure braids} is an immediate consequence of
Propositions \ref{prop7.1} and \ref{prop7.2},  Lemma \ref{lem6.1}, and the Main Theorem for irreducible braids. \hfill $\Box$

\medskip

\begin{rem}\label{rem5.1}  The mapping class $\widehat{\mathfrak{m}_{b,\infty}}$ can be recovered from the irreducible nodal components of the braid up to a product of Dehn twists about simple closed curves that are homologous to the admissible curves.
\end{rem}
Indeed, this statement was obtained by the proof of inequality  \eqref{eq6.52}.

\chapter[General case. Irreducible components.  Proof of Main Theorem.]   {The general case. Irreducible nodal components, irreducible braid components, and the proof of the Main Theorem.}
\label{chapter7a}

\setcounter{equation}{0}

In this chapter we prove the Main Theorem for reducible non-pure braids.  While the strategy of the proof in this case is similar to that in the case of pure braids, there are some difficulties that need additional considerations.
For instance, for each conjugacy class of pure braids the irreducible braid components can be obtained by forgetting strands. This approach does not work for non-pure braids.
We define irreducible braid components and irreducible nodal components in the non-pure case, and describe the respective decompositions into irreducible components and the recovery procedures.

\section {Irreducible nodal components. The general case.}\label{sec:7a.1}
In this Chapter 
we consider a reducible, not necessarily pure, braid
$b \in {\mathcal B}_n$ and  its mapping class $\mathfrak{m}_b\in {\mathfrak M} (\overline{\mathbb{D}} ; \partial \mathbb{D} , E_n)$. 
We represent the mapping class
${\mathfrak m}_{b,\infty}=\mathcal{H}_{\infty}(\mathfrak{m}_b) \in {\mathfrak M} ({\mathbb{P}^1} ; \infty , E_n)$ (see  equality \eqref{eq2.8}) 
by a self-homeomorphism  $\varphi_{b,\infty}$ of $\mathbb{P}^1$ 
which is the identity outside the unit disc $\mathbb{D}$. We may choose  $\varphi_{b,\infty}$ so that 
it is completely reduced by
an admissible system of curves  ${\mathcal C} = \{ C_1 ,
\ldots , C_k\}$  in ${\mathbb D}
\backslash E_n \subset \mathbb{P}^1\backslash E'_n $. In particular,
$\varphi _{b,\infty}$ leaves the union $\underset{C \in\mathcal{C}}{\bigcup} C$
invariant, and also leaves the complement ${\mathbb P}^1 \backslash
\underset{C \in \mathcal{C}}{\bigcup}C$ invariant.
The nodal components and the irreducible braid components will be associated to the isotopy class of the chosen admissible system of curves.

The connected components of $\mathbb{P}^1\setminus \bigcup_{C \in \mathcal{C}}C$ are labelled as follows. 
As in Chapter \ref{chapter6} 
the connected component of $\mathbb{P}^1\setminus \bigcup_{C \in \mathcal{C}}C$ that contains $\infty$ is denoted by $S^{1,1}$ and is said to be of generation $1$.
For $\ell\geq 2$ the interior of the union of the closure of all components of generation not exceeding $\ell-1$ is denoted by $Q_{\ell-1}$. The components of $\mathbb{P}^1\setminus \bigcup_{C \in \mathcal{C}}C$ that share a boundary component with $Q_{\ell-1}$ are called the components of generation $\ell$. 
However,
since $b$ is not required to be a pure braid, the connected components 
of $\mathbb{C} \setminus \mathcal{C}$
of generation $\ell$ are not necessarily fixed by $\varphi_{b,\infty}$. But the components of each generation $\ell$ are permuted  along cycles. \index{$\varphi_{b,\infty}$}
Indeed, since $\varphi$ fixes $\infty \in Q_{\ell-1}$ it fixes $Q_{\ell-1}$ setwise and, hence, permutes its boundary components, and therefore permutes the components of  ${\mathbb P}^1
\backslash \underset{C \in \mathcal{C}}{\bigcup} C$ that
share a boundary component with $Q_{\ell-1}$.

We label the connected components of  ${\mathbb P}^1
\backslash \underset{C \in\mathcal{C}}\bigcup C$ of generation $\ell\geq 2$ as follows. Choose for each cycle 
of components of  ${\mathbb P}^1
\backslash \underset{C \in\mathcal{C}}\bigcup C$ of generation $\ell$ a set and denote it by  $S_1^{\ell,i}$. 
(The choice will be made in Section \ref{sec:6.4} in the proof of Lemma 5.7. Here the choice of the label does not play a role. We just assume that some label is chosen.)
Let $k(\ell,i)$ be the length of the cycle,  i.e. the smallest positive integer for which $\varphi_{b,\infty} ^{k(\ell,i)}$ maps $S_1^{\ell,i}$ onto itself. Put
\begin{equation}\label{eq7a.1'}
S_{j+1}^{\ell,i} \stackrel{def}=\varphi_{b,\infty} (S_j^{\ell,i})  \, ,\;\; j = 1,\ldots
, k(\ell,i)-1 .
\end{equation}
For all $\ell,\, $ and $i$
\index{$S_j^{\ell,i}$} \index{$\partial_{\mathfrak E}{S_j^{\ell,i}}$ } \index{$\partial_{\mathfrak I}{S_j^{\ell,i}}$ }
\begin{equation}\label{7a.1''}
\varphi_b (S_{k(\ell,i)}^{\ell,i})=S_1^{\ell,i} ,
\end{equation}
The $\varphi_{b,\infty}$-cycles of connected components  of generation $\ell$ of ${\mathbb P}^1
\backslash \underset{C \in\mathcal{C}}\bigcup C$ are denoted
by ${\rm cyc}_0^{\ell,i}= \big( S_1^{\ell,i}\,,\,\ldots\,,\, S_{k(\ell,i)}^{\ell,i}\big)    \,,\; i=1,\ldots,k_{\ell}$,
$$
S_1^{\ell,i} \overset{\varphi_{b,\infty}}{\longrightarrow} S_2^{\ell,i}
\overset{\varphi_{b,\infty}}{\longrightarrow} \ldots
\overset{\varphi_{b,\infty}}{\longrightarrow} S_{k(\ell,i)}^{\ell,i}
\overset{\varphi_{b,\infty}}{\longrightarrow} S_1^{\ell,i} \, .
$$
Note that the number of cycles $k_{\ell}$ does not exceed the number of components $k_{\ell}'$
of generation $\ell$. \index{ $k_{\ell}$} 
\index{$S^{\ell,i}_j$} \index{$k(\ell,i)$}

The set of distinguished points ${E'}_j^{\ell,i}=S_j^{\ell,i}\cap E'_n$ is permuted in the same way:
$$
{E'}_1^{\ell,i} \overset{\varphi_{b,\infty}}{\longrightarrow} {E'}_2^{\ell,i}
\overset{\varphi_{b,\infty}}{\longrightarrow} \ldots
\overset{\varphi_{b,\infty}}{\longrightarrow} {E'}_{k(\ell,i)}^{\ell,i}
\overset{\varphi_{b,\infty}}{\longrightarrow} {E'}_1^{\ell,i} \, .
$$
The cycles of connected components of  ${\mathbb P}^1
\backslash \Big((\underset{C \in\mathcal{C}}\bigcup C) \bigcup E_n'\Big)$ of generation $\ell$ are denoted 
by ${\rm cyc}^{\ell,i}=\big( S_1^{\ell,i}\setminus E_n'\,,\,\ldots\,,\, S_{k(\ell,i)}^{\ell,i}\setminus E_n'\big) ,\, i=1,\ldots,k_{\ell}$.

Consider the $\varphi_{b,\infty}$-cycle of curves  $\partial_{\mathfrak{E}}S^{\ell,i}_j$, $j=1,\ldots,k(\ell,i),$ of generation $\ell$ and the cycle of discs
$\delta^{\ell,i}_j$ in $\mathbb{C}$ that are bounded by $\partial\delta^{\ell,i}_j = \partial_{\mathfrak{E}}S^{\ell,i}_j$. The homeomorphism  $\varphi_{b,\infty}$ permutes the discs along the cycle $(\delta_1^{\ell,i},\delta_2^{\ell,i},\ldots,
  \delta_{k(\ell,i)}^{\ell,i})$,
$$
\delta_1^{\ell,i} \overset{\varphi_{b,\infty}}{\longrightarrow} \delta_2^{\ell,i}
\overset{\varphi_{b,\infty}}{\longrightarrow} \ldots
\overset{\varphi_{b,\infty}}{\longrightarrow} \delta_{k(\ell,i)}^{\ell,i}
\overset{\varphi_{b,\infty}}{\longrightarrow} \delta_1^{\ell,i} \, .
$$

Notice that a cycle of discs $\delta_j^{\ell,i}$ of length $k(\ell,i)$ corresponding to a cycle of components of ${\mathbb P}^1
\backslash \underset{C\in \mathbb{P}^1}C$
may not contain a cycle of distinguished points of this length.
However, the following lemma holds.
\begin{lemm}\label{lem7a.1} Let $\ell >1$. Take any $\varphi_b$-cycle ${\rm cyc}^{\ell,i}$ of sets $S_j^{\ell,i}, \,j=1,\ldots k({\ell},i)$ of length $k(\ell,i)$, and the associated cycles of discs $\delta_j^{\ell,i}$ and of curves $\partial_{\mathfrak{E}}{S_{j}^{\ell,i}}$. Then there is an $\mathfrak{M}(\mathbb{P}^1; \mathbb{P}^1\setminus\mathbb{D}, E_n)$-isotopy of $\varphi_{b,\infty}$, that changes $\varphi_{b,\infty}$ only in a small neighbourhood of $\partial_{\mathfrak{E}}{S_1^{\ell,i}}$, and
provides a homeomorphism $\varphi_{b,\infty}'$ which has a 
cycle of perhaps non-distinguished points $\mathbold{z}_j^{\ell,i} \in S_j^{\ell,i} \subset \delta_j^{\ell,i}$ of length $k({\ell},i)$, 
$\varphi'_{b,\infty}(\mathbold{z}_j^{\ell,i})=\mathbold{z}_{j+1}^{\ell,i},
\,j=1,\ldots, k(\ell,i)-1,\;\;
\varphi'_{b,\infty}(\mathbold{z}_{k(\ell,i)}^{\ell,i})=
\mathbold{z}_1^{\ell,i}\,.$ 
Moreover, the new homeomorphism $\varphi'_{b,\infty}$ has the following property.
\begin{equation}\label{eq7a.1}
\mbox{For} \;\, 
1\leq j \leq
k(\ell,i), \;
{\varphi'}_{b,\infty}^{k(\ell,i) } \;
\mbox{fixes a neighbourhood 
of }\; \partial_{\mathfrak E} \, S_j^{\ell,i} \cup \{
\mathbold{z}_j^{\ell,i} \}  \, \mbox{pointwise}  .
\end{equation}
\end{lemm}
We will apply Lemma  \ref{lem7a.1} to all cycles of discs $\delta_j^{\ell,i}$.
Recall that the collection of the $\delta_j^{\ell,i}, \, i=1,\ldots, k_{\ell},\, j=1,\ldots k(\ell,i),$  is the set of all holes of $Q_{\ell -1}$. Hence, there is a one-to-one correspondence between the set of holes of $Q_{\ell -1}$ and the points $\mathbold{z}_j^{\ell,i}$ contained in the holes.

\medskip

\noindent {\bf Proof of Lemma \ref{lem7a.1}.} With the requirement that $\varphi'_{b,\infty}=\varphi_{b,\infty}$ in a neighbourhood of   $\partial_{\mathfrak E} \, S_j^{\ell,i},\, j=2,\ldots,k(\ell,i),$ the equation
\begin{equation}\label{eq7.51'}
\varphi'_{b,\infty}=\varphi_{b,\infty}^{1-k(\ell,i)}
\end{equation}
in a neighbourhood of $\partial_{\mathfrak E} \, S_{1}^{\ell,i}$ implies ${\varphi'}_{b,\infty}^{k(\ell,i)}=\rm{Id}$ in a neighbourhood of  $\partial_{\mathfrak E} \, S_{1}^{\ell,i}$. Then this equation also holds in a neighbourhood of  $\partial_{\mathfrak E} \, S_{j}^{\ell,i}$, $j=2,\ldots,k(\ell,i),$ since the restriction of ${\varphi'}_{b,\infty}^{k(\ell,i)}$
to a neighbourhood of
$\partial_{\mathfrak E} \, S_{j}^{\ell,i}$ equals $ {\varphi}_{b,\infty}^{j-2} \circ\varphi'_{b,\infty}\circ \varphi_{b,\infty}^{k(\ell,i)-j+1}$. It is clear that there is an
$\mathfrak{M}(\mathbb{P}^1; \mathbb{P}^1\setminus\mathbb{D}, E_n)$-isotopy of $\varphi_{b,\infty}$ that changes $\varphi_{b,\infty}$ only in a small neighbourhood of $\partial_{\mathfrak{E}}{S_1^{\ell,i}}$
and provides a homeomorphism $\varphi_{b,\infty}'$ that satisfies \eqref{eq7.51'} in a neighbourhood of $\partial_{\mathfrak E} \, S_{1}^{\ell,i}$.

\index{$\mathbold{z}_1^{\ell,i}$}
It remains to choose a point $\mathbold{z}_1^{\ell,i}\in \delta_1^{\ell,i}$ that is contained in this neighbourhood. The homeomorphism $\varphi_{b,\infty}'$   maps $\mathbold{z}_1^{\ell,i}\in \delta_1^{\ell,i}$
along the required cycle and the $k(\ell,i)$-th iterate of  $\varphi_{b,\infty}'$ fixes a neighbourhood of $\mathbold{z}_1^{\ell,i}$.
The lemma is proved. \hfill $\Box$
\medskip

In the following the homeomorphism $\varphi_{b,\infty}$ will satisfy 
all conditions 
that are obtained for the homeomorphism $\varphi_{b,\infty}'$
of Lemma  \ref{lem7a.1}.

\medskip

As in the case of pure braids (see also Section \ref{sec:6.1} and also Section \ref{sec:2.4}) we associate to $b$ a nodal surface $Y$ and an isotopy class 
${\mathfrak{m}}_{b,\odot}$ of self-homeomorphisms of $Y$, using
a continuous surjection $w: \mathbb{P}^1\setminus E'_n \to Y$
whose preimage of each node is a curve 
$C \in\mathcal{C}$.

Recall that the set of nodes of $Y$ is denoted by $\mathcal{N}$. The surjection $w$ maps $\mathbb{P}^1\setminus (E'_n \cup \underset{C \in \mathcal{C}}\bigcup C)$ homeomorphically onto $Y \setminus \mathcal{N}$. 
Let $Y_j^{\ell,i}$ be the component $w(S_j^{\ell,i}\setminus E'_n)$   of $Y\setminus \mathcal{N}$.
Conjugate $\varphi_{b,\infty}\mid \mathbb{P}^1\setminus(E'_n \cup \underset{C \in \mathcal{C}}\bigcup C)$ by the inverse of
$w\mid \mathbb{P}^1\setminus(E'_n \cup \underset{C \in \mathcal{C}}\bigcup C)$ and denote the conjugate by ${\varphi_{b,\odot}}$. Its isotopy class 
of self-homeomorphisms of $Y\setminus \mathcal{N}$ is denoted by ${\mathfrak{m}}_{b,\odot}$. The conjugacy
class $\widehat{{\mathfrak{m}}_{b,\odot}}$ is called the nodal conjugacy class associated to  $\mathfrak{m}_{b,\infty}$ and the system of admissible curves $\mathcal{C}$.
\index{$\varphi_{b,\odot}$}

Notice that the mapping $w$ extends across the punctures to a continuous surjection from $\mathbb{P}^1$ onto the compact nodal surface $Y^c$ obtained by filling the punctures of $Y$.  The extended mapping is also denoted by $w$. The elements  of the class ${\mathfrak{m}}_{b,\odot}$ extend across the punctures of the nodal surface $Y$ to self-homeomorphisms of the compact nodal surface $Y^c$. By an abuse of notation we denote the class of extended mappings also by ${\mathfrak{m}}_{b,\odot}$ and denote also the respective conjugacy class by  $\widehat{{\mathfrak{m}}_{b,\odot}}$.

The mapping ${\varphi_{b,\odot}}$ permutes the nodal components  $Y_j^{\ell,i}$ of $Y$ along cycles corresponding to the cycles of the $S_j^{\ell,i}\setminus {E'}_j^{\ell,i}$. We denote the cycles of nodal components  $Y_j^{\ell,i}$ by 
${\rm cyc}^{\ell , i}_{b,\odot}$ and call them the nodal cycles. 
\index{${\rm cyc}^{\ell , i}_{b,\odot}$}
\index{${\rm cyc}^{\ell , i}_{b,\odot}$}

The conjugacy classes  $\widehat{{\mathfrak{m}}^{\ell,i}_{b,\odot}}$       \index{$\widehat{{\mathfrak{m}}^{\ell,i}_{b,\odot}}$}
of the restrictions  ${{\mathfrak{m}}^{\ell , i}_{b,\odot}}$     of ${\mathfrak{m}}_{b,\odot}$ 
to the cycles ${\rm cyc}^{\ell , i}_{b,\odot}$,
are called the irreducible nodal components of the class
$\widehat{{\mathfrak m}_{b,\infty}}$.  \index{irreducible nodal component}
The irreducible
nodal components determine the class $\widehat{{\mathfrak
m}_{b,\infty}}$ only up to products of powers of some Dehn twists. This can be
seen in the same way as in the case of pure braids. (See later the Remark \ref{rem6.2}.)
\index{nodal cycle}

Instead of the nodal components $Y^{\ell,i}_j$ we may, equivalently, consider the respective cycles of compact Riemann surfaces $(Y^{\ell,i}_j)^c$
and extend the homeomorphisms betwen the $Y^{\ell,i}_j$ to homeomorphisms  between the compact surfaces. The obtained conjugacy class we will also denote by  $\widehat{\mathfrak{m}_{b,\odot}^{\ell,i}}$.

If the length of the cycle is bigger than one, the class $\widehat{\mathfrak{m}_{b,\odot}^{\ell,i}}$ is the  conjugacy class of a mapping class on a not connected Riemann surface.
Lemma \ref{lemm6.5} below states that each irreducible nodal component $\widehat{{\mathfrak{m}}^{\ell , i}_{b,\odot}}$ is determined by a conjugacy class of mapping classes of self-homeomorphisms of a connected Riemann surface. 
For instance, for the length $k(\ell,i)$ of the cycle  ${\rm cyc}^{\ell , i}_{b,\odot}$,  we take the $k(\ell,i)$-th power of $\varphi_{b,\odot}$, and restrict it to a component of the cycle, say to $Y_1^{\ell,i}$. The lemma states that the conjugacy class
$\reallywidehat{{\mathfrak m}_{b^{k(\ell,i)},\odot} \mid
Y_1^{\ell,i}}$ of the restricted mapping determines the irreducible nodal component ${{\mathfrak{m}}^{\ell , i}_{b,\odot}}$.  Notice that the mappings
$\varphi_{b,\odot}^{k(\ell,i)}\mid Y_j^{\ell,i}\,, j=1,\ldots,k(\ell,i)\,,$ are conjugate, hence, the class  $\reallywidehat{{\mathfrak m}_{b^{k(\ell,i)},\odot} \mid
Y_1^{\ell,i}}$ is defined by the cycle and is independent on the choice of the set in the cycle to which the mapping $\varphi_{b,\odot}^{k(\ell,i)}$ is restricted.

The key ingredient of the proof of the following Lemma \ref{lemm6.5} is a simple Lemma on Conjugation, which was used by Bers \cite{Be1}. For convenience of the reader we give a proof in the Appendix B.
\begin{lemm}\label{lemm6.5}
Suppose a self-homeomorphism $\psi_{\odot}$ of $\,Y^{\ell,i}_1\,$ represents the
class  $\,\reallywidehat{{\mathfrak m}_{b^{k(\ell,i)},\odot} \mid
Y_1^{\ell,i}} \;$ for the $\,k(\ell,i)$-th power of the braid $\,b\,$. Then any
self-homeomorphism $\,\varphi_{\odot}\,$ of the union $Y_1^{\ell,i}\cup\ldots\cup Y_{k(\ell,i)}^{\ell,i}\,,$
that moves the sets $Y_j^{\ell,i}$ along the cycle
\begin{equation}\label{eq7a.30}
Y_1^{\ell,i} \overset{\varphi_{\odot}}{\longrightarrow} Y_2^{\ell,i}
\overset{\varphi_{\odot}}{\longrightarrow} \ldots
\overset{\varphi_{\odot}}{\longrightarrow} Y_{k(\ell,i)}^{\ell,i}
\overset{\varphi_{\odot}}{\longrightarrow} Y_1^{\ell,i} \,,
\end{equation}
and satisfies the condition
\begin{equation}\label{eq7a.31}
\varphi_{\odot}^{k(\ell,i)}\mid Y_1^{\ell,i}= \psi_{\odot},
\end{equation}
represents $\reallywidehat{{\mathfrak
m}_{b,\odot}\mid {\rm cyc}^{\ell,i}_{\odot}}$.

Vice versa, for any self-homeomorphism $\varphi_{\odot}$ of
$Y_1^{\ell,i}\cup\ldots\cup Y_{k(\ell,i)}^{\ell,i}$ that represents $\reallywidehat{{\mathfrak
m}_{b,\odot}\mid {\rm cyc}^{\ell,i}_{\odot}}$
the restriction
$\varphi_{\odot}^{k(\ell,i)}\mid Y_1^{\ell,i}$
represents  $\,\reallywidehat{{\mathfrak m}_{b^{k(\ell,i)},\odot} \mid
Y_1^{\ell,i}} $.
\end{lemm}
\index{$Y_j^{\ell,i}$}
\bigskip

\noindent {\bf Proof.} It is clear that for each $\varphi_{\odot} \in \reallywidehat{{\mathfrak
m}_{b,\odot} \mid {\rm cyc}^{\ell,i}_{\odot}}$
the inclusion $\varphi_{\odot} ^{k(\ell,i)}\mid Y_1^{\ell,i}
\in
\reallywidehat{{\mathfrak m}_{b^{k(\ell,i)},\odot} \mid
Y_1^{\ell,i}}$ holds.

Prove the first part of the lemma.
Let $\psi_{\odot}$ in  $\reallywidehat{{\mathfrak m}_{b^{k(\ell,i)},\odot}
\mid Y_1^{\ell,i}}$. Define a self-homeomorphisms $\varphi_{\odot}
\in \reallywidehat{{\mathfrak m}_{b,\odot} \mid
{\rm cyc}^{\ell,i}_{\odot}}$
as follows. For each
$j = 1,\ldots , k(\ell,i)-1$, we take any homeomorphism
$\varphi_{\odot,j}$ from $Y_j^{\ell,i}$ onto
$Y_{j+1}^{\ell,i}$. For $j = k(\ell,i)$ we take the homeomorphism
$\varphi_{\odot,k(\ell,i)}$ from $Y_{k(\ell,i)}^{\ell,i}$
onto $Y_1^{\ell,i}$ for which
\begin{equation}
\label{eq6.68} \varphi_{\odot,k(\ell,i)} \circ \ldots \circ
\varphi_{\odot,1} = \psi_{\odot} \quad {\rm on} \quad
Y_1^{\ell,i} \, .
\end{equation}
Consider the self-homeomorphism $\varphi_{\odot}$ of
$\underset{j=1}{\overset{k(\ell,i)}\bigcup} \ Y_j^{\ell,i}$ which
equals $\varphi_{\odot,j}$ on $Y_j^{\ell,i}$. 
Then
$$
\varphi^{k(\ell,i)}_{\odot} \mid Y_1^{\ell,i} =
\psi_{\odot} \, .
$$

We have to prove that any mapping $\varphi_{\odot}$ obtained in this way
is in the class $\reallywidehat{{\mathfrak m}_{b,\odot}
\mid {\rm cyc}^{\ell,i}_{\odot}}$.
Let $w$ be as before the continuous surjection from $\mathbb{P}^1\setminus E'_n$ onto $Y$ that is a homeomorphism from $\mathbb{P}^1\setminus (E'_n \cup \underset{C \in \mathcal{C}}\bigcup C)$ onto $Y\setminus \mathcal{N}$.
Let $w^{\ell,i}$ be the restriction of $w$ to  $\underset{j=1}{\overset{k(\ell,i)}\bigcup}  S_j^{\ell,i} \backslash
E'_n $. The mapping $w^{\ell,i}$ takes
$S^{\ell,i}_j\setminus E_n'$ homeomorphically onto $Y^{\ell,i}_j$ for each $j$. Conjugate 
$\varphi_{b,\infty}\mid  \underset{j=1}{\overset{k(\ell,i)}\bigcup}  S_j^{\ell,i} \backslash
E'_n    $ by the inverse of $w^{\ell,i}$ and
denote the obtained self-homeomorphism
of $\underset{j=1}{\overset{k(\ell,i)}\bigcup} Y^{\ell,i}_j$
by $\tilde{\varphi}_{\odot}$. Then  $\tilde{\varphi}_{\odot}$ represents the class $\reallywidehat{{\mathfrak
m}_{b,\odot}\mid {\rm cyc}^{\ell,i}_{\odot}}$. 
The mappings
$\varphi_{\odot}$ and $\tilde{\varphi}_{\odot}$ 
move the $Y_j^{\ell,i}$ along the same cycle, 
and both, 
$\tilde{\varphi}^{k(\ell,i)}_{\odot} \mid Y_1^{\ell,i}$ and ${\varphi}^{k(\ell,i)}_{\odot} \mid Y_1^{\ell,i}$, represent the same
conjugacy class  $\,\reallywidehat{{\mathfrak m}_{b^{k(\ell,i)},\odot} \mid
Y_1^{\ell,i}} $. Since $\tilde{\varphi}_{\odot}$ represents
$\reallywidehat{{\mathfrak
m}_{b,\odot}\mid {\rm cyc}^{\ell,i}_{\odot}}$,
by the Lemma on
Conjugation (see Appendix B) also the mapping $\varphi_{\odot}$ represents
this class. 
The lemma is proved. \hfill $\Box$

\begin{rem} \label{rem6.1}
The restriction $\varphi_{b,\infty}^{k(\ell,i)}\mid S_1^{\ell,i}\setminus E'_n$, equivalently, the mapping $\varphi_{b,\odot}^{k(\ell,i)}
\mid Y_1^{\ell,i}$, is irreducible.
\end{rem}
Suppose the contrary. Then there exists an $\mathfrak{M}(\mathbb{P}^1\setminus E'_n)$-isotopy that joins $\varphi_{b,\infty}^{k(\ell,i)}$ 
with a self-homeomorphism $\psi$ of $\mathbb{P}^1$
and changes the values of $\varphi_{b,\infty}^{k(\ell,i)}$ only on $S^{\ell,i}_1$, such that $\psi$ has the following property. 
It fixes a simple closed curve $\gamma$ in $S_1^{\ell,i}$ that is not homotopic in $\mathbb{P}^1\setminus E'_n$ to a point, or to one of the punctures $S_1^{\ell,i}\cap E'_n$ in $S_1^{\ell,i}$, or to a boundary component of
$S_1^{\ell,i}$. Then the mapping $\varphi_{b,\infty}'$, that equals $\varphi_{b,\infty}$ on $\mathbb{P}^1 \setminus S_1^{\ell,i}$ (in particular, on $S_2^{\ell,i}\cup\ldots S_{k(\ell,i)}^{\ell,i}$) and equals $(\varphi_{b,\infty})^{-k(\ell,i)+1}\circ\psi$ on $S_1^{\ell,i}$, fixes the set of loops $\gamma\cup \varphi_{b,\infty}'(\gamma)\cup\ldots\cup (\varphi_{b,\infty}')^{k(\ell,i)-1}(\gamma)$ setwise and is $\mathfrak{M}(\mathbb{P}^1\setminus E'_n)$-isotopic to $\varphi_{b,\infty}$ by the Lemma on conjugation.
This contradicts the fact that the system of curves $\mathcal{C}$ was maximal.

\section{Irreducible braid components. The general
case.}\label{sec:7a.2}

We continue to consider non-pure reducible braids $b$ with base point $E_n\subset \mathbb{D}$ 
together with the associated homeomorphism $\varphi_{b,\infty} \in {\rm Hom}^+( \mathbb{P}^1 ;
\mathbb{P}^1\setminus \mathbb{D}\,,E_n)$, that fixes an admissible system of curves $ \underset{C \in\mathcal{C}}\bigcup C$ setwise, is completely reduced by this system, and satisfies \eqref{eq7a.1} for each cycle ${\rm cyc}_0^{\ell,i}$ of connected components of  ${\mathbb P}^1
\backslash \underset{C \in\mathcal{C}}\bigcup C$. Let $\varphi_t \in {\rm Hom}^+ (\mathbb{P}^1 ;
{\mathbb{P}^1 \setminus {\mathbb D}})$, $t \in [0,1]$, be a continuous family of self-homeomorphisms of $\mathbb{C}$
such that $\varphi_0 = \rm{id}$ and $\varphi_1 = \varphi_{b,\infty}$. 
The geometric braid $\{(t,\varphi_t(E_n) ), \, t \in [0,1]\}$ is contained in the cylinder $[0,1]\times \mathbb{D}$ and represents $b$.
Put $g(t)= \varphi_t(E_n),\; t \in [0,1]$,
$g:[0,1] \to C_n(\mathbb{C}) \diagup {\mathcal{S}_n }.\,$

\bigskip

\noindent {\bf The outermost braid.}
The tubular braid associated to the outermost component $S^{1,1}$ can be defined 
similarly as in Chapter \ref{chapter6}.
Consider the set of all holes of the outermost component. We give each hole the label of the unique $S^{2,i}_j$ whose exterior boundary is the boundary of the hole, and denote the respective hole by $\delta^{2,i}_j$.

Associate the tubular geometric braid
\begin{align}
\label{eq7a.2} \Bigg\{\Big(t, \varphi_t (E^{1,1} \cup \bigcup_{i=1}^{k_2}  \,\bigcup_{j=1}^{k(2,i)}\,
\overline{\delta_j^{2,i}}) \Big) \, , \quad t \in [0,1] \Bigg\}\,
\end{align}
to $b$ and $S^{1,1}$. The set correctly defines a tubular braid, since $\varphi_0$ is the identity and $\varphi_1= \varphi_{b,\infty}$ maps the set   $E^{1,1} \cup \bigcup_{i=1}^{k_2}  \,\bigcup_{j=1}^{k(2,i)}\,
\overline{\delta_j^{2,i}} $ onto itself.

We define a geometric fat braid that is a deformation retract of the tubular braid. Recall that a cycle ${\rm cyc}_0^{2,i}$ of length $k(2,i)$ may not contain a cycle of points of $E_n$ of the same length. By this reason for each cycle ${\rm cyc}_0^{2,i}$ we choose the cycle of points $\mathbold{z}_j^{2,i}\in S_j^{2,i}$  chosen in Lemma \ref{lem7a.1} (see also relation \eqref{eq7a.1}) as initial points of the fat strands.
In other words, let $\mathbold{ E}^{1,1}$ be the collection of all points
$\mathbold{z}_j^{2,i} \in S_j^{2,i} \subset \delta_j^{2,i}\;$, $i = 1,\ldots , k_2\;$, $j = 1,\ldots , k(2,i)\;$, assigned to the set of holes of $S^{1,1}$ in Lemma \ref{lem7a.1},
\begin{equation}\label{eq7a.4}
{\mathbold{E}}^{1,1}= \bigcup_{i=1}^{k_2} \{\mathbold{z}_1^{2,i}, \ldots, \mathbold{z}_{k(2,i)}^{2,i}\}\,.
\end{equation}
Associate to $S^{1,1}$ the geometric fat braid
\begin{align}
\label{eq7a.4''}
\Big\{\big(t,\varphi_t (E^{1,1} \cup \mathbold{ E}^{1,1})\big) \, , \quad t \in
[0,1]\Big\} \, .
\end{align}
Let $\mathbold{B}(1,1)$ be the isotopy class of the geometric fat braid \eqref{eq7a.4''}, and let $\widehat {\mathbold{{B}}(1,1)}$ be its conjugacy class.

Denote by $n(1,1)$ the number of points in $E^{1,1} \cup \mathbold{ E}^{1,1}$. We
write the mapping defining the geometric fat braid \eqref{eq7a.4''} as a fat map $\mathbold{g}^{1,1}$,
\begin{equation}
\label{eq7a.5'} \mathbold{g}^{1,1}: [0,1] \to C_{n(1,1)} ({\mathbb C}) \diagup
{\mathcal S}_{n(1,1)} \, .
\end{equation}
See Figure \ref{fig7a.2}, where $E^{1,1} = \emptyset$. In this figure the
connected components of ${\mathbb C} \backslash \underset{C \in
{\mathcal C}}{\bigcup} C$ of generation $2$ are $S_1^{2,1} ,
S_2^{2,1} , S_3^{2,1}$, which are moved along a $3$-cycle by
$\varphi_{b,\infty}$. The set $\mathbold{ E}^{1,1}$ equals $\mathbold { E}^{1,1} =
\{\mathbold {z}_1^{2,1} , \mathbold {z}_2^{2,1} , \mathbold {z}_3^{2,1} \}$.
The following equality holds
\begin{lemm}\label{lemm6.2}\begin{equation}
\label{eq6.31}\reallywidehat{{\mathfrak m}_{B(1,1),\infty}}= \reallywidehat{{\mathfrak m}_{b,\odot}^{1,1}}\,.
\end{equation}
\end{lemm}
\smallskip

\noindent {\bf Proof of Lemma \ref{lemm6.2}.} The
situation differs from that of Lemma \ref{lem6.1} by the fact that $b$ is not
necessarily a pure braid and, hence, the interior boundary
components of $S^{1,1}$ may be permuted by the mapping $\varphi_{b,\infty}$.
Again, the mapping $\varphi_{b,\infty}\mid S^{1,1}\setminus {E'}^{1,1}$ 
represents the class
$\reallywidehat{{\mathfrak m}_{b,\odot}^{1,1}}\in\widehat{\mathfrak{M}}(Y^{1,1}) $. The mapping $\varphi_{b,\infty}\mid \mathbb{P}^1\setminus(\mathbold{E}^{1,1} \cup {E'}^{1,1})$ 
represents the class  $\reallywidehat{{\mathfrak m}_{B(1,1),\infty}} \in \widehat{\mathfrak{M}}(\mathbb{P}^1\setminus(\mathbold{E}^{1,1} \cup {E'}^{1,1}))     $. (Recall that ${E'}^{1,1}=E^{1,1}\cup \{\infty\}$.)

Similarly as in the proof of Lemma \ref{lem6.1} we take for each first labeled set $\delta^{2,i}_1$ of a cycle whose exterior boundary is an interior boundary component of $S^{1,1}$ an open annulus $A^{2,i}_1$ that is contained in $S^{1,1}\setminus {E'}^{1,1}$ and shares a boundary component with the boundary of $\delta^{1,1}_1$. Moreover we require that  $A^{2,i}_1$ is contained in the neighbourhood of $\partial \delta^{2,i}_1$ which is fixed by the iterate $\varphi_{b,\infty}^{k(2,i)}$. We put $A^{2,i}_j= \varphi_{b,\infty}^{j-1}(A^{2,i}_1),\; j=2,\ldots,k(2,1).$ The annuli are taken small enough so that all obtained annuli for $i=1,\ldots,k_2$, $j=1,\ldots,k(2,i)$ are pairwise disjoint.

There is a homeomorphism $\tilde{ w}^{1,1}$ from ${S^{1,1}}\setminus {E'}^{1,1}$ onto $\mathbb{P}^1\setminus ( {E'}^{1,1} \cup \mathbold{E}^{1,1})$ with the following properties. $\tilde{ w}^{1,1}$  is equal to the identical injection on the set $S^{1,1}\setminus ({E'}^{1,1} \cup\bigcup A^{2,i}_j)$.
Moreover, for each $j$ the mapping $\tilde{ w}^{1,1}$ takes each annulus $A^{2,i}_j$ onto the punctured disc ${\delta'}^{2,i}_j \setminus \{\mathbold{z}^{2,i}_j\}$, where
${\delta'}^{2,i}_j\stackrel{def}=A^{2,i}_j \cup \overline{\delta^{2,i}_j}$. 
Conjugate $\varphi_{b,\infty}\mid S^{1,1}\setminus {E'}^{1,1}$
with the inverse of $\tilde {w}^{1,1}$. The conjugate $\tilde{\varphi}^{1,1}_{b,\infty}$ is related to $\varphi_{b,\infty}^{1,1}\stackrel{def}=  \varphi_{b,\infty}\mid \mathbb{P}^1\setminus(\mathbold{E}^{1,1} \cup {E'}^{1,1})$ by isotopy and conjugation.
Indeed, the two mappings, $\varphi_{b,\infty}\mid S^{1,1}\setminus {E'}^{1,1}$ and $\tilde{\varphi}_{b,\infty}^{1,1}$, 
differ only on the punctured discs ${\delta'}^{2,i}_j\setminus \mathbold{z}^{2,i}_j$ around the points $\mathbold{z}^{2,i}_j$ of $\mathbold{E}^{1,1}$ and map the punctured discs along the same cycles. They are equal on the boundary of each ${\delta'}^{2,i}_j$.
Moreover, for each $i$ the $k(2,i)$-th iterate of both mappings is the identity on each $\partial {\delta'}^{2,i}_j$, $j=1,\ldots, k(i,j)$, 
since by \eqref{eq7a.1}  $\varphi_{b,\infty}^{k(2,i)}$ is the identity on each $A^{2,i}_1$. 
Apply for each $i$ the Lemma on conjugation to the restriction of the two mappings
${\varphi}_{b,\infty}^{1,1}$ and $\tilde{\varphi}_{b,\infty}^{1,1}$
to $\bigcup_j^{k(2,i)}{\delta'}_j^{2,i}$, and take into account the following fact. If a self-homeomorphism of a closed disc punctured at a point fixes the boundary circle pointwise, then it is isotopic to the identity through self-homeomorphisms that fix the boundary circle pointwise.
We proved that ${\varphi}_{b,\infty}^{1,1}$ and $\tilde{\varphi}_{b,\infty}^{1,1}$ 
are related by isotopy and conjugation. \hfill $\Box$

\bigskip
\noindent{\bf The irreducible braid components $\widehat{\mathbold{B}(\ell,i)}$ associated to the
$\varphi_{b,\infty}$-cycles $\rm {cyc}^{\ell,i}$, $\ell >1$.} 
Recall that, for instance, the ordinary part of the geometric fat braid representing $\mathbold{B}(1,1)$ was found by
forgetting all strands of the original braid with initial points not in $E^{1,1}=E_n\cap S^{1,1}$. It was important that the set $E^{1,1}$ is invariant under $\varphi_{b,\infty}$.

In order to apply this recipe for getting the irreducible braid component  $\widehat{\mathbold{B}(\ell,i)}$ for $\ell>1$ we wish to start with a $\varphi_{b,\infty}$-invariant set of points. The intersection of $E_n$ with the union of the sets $S^{\ell,i}_j$ of the cycle $\mbox{cyc}^{\ell,i}$ has this property.
The problem
is that the geometric braid obtained by forgetting all strands with initial point
not contained in
this set carries  also information about the tubular braid $\big(t, \varphi_t(\bigcup_j\delta^{\ell,i}_j)\big)$, which is part of the information carried by the geometric fat braids of generation less than $\ell$.
However, the power $\varphi_{b,\infty}^{k(\ell,i)}$ leaves each $S^{\ell,i}_j$ , $j=1,\ldots,k(\ell,i)$, invariant and, for instance, the geometric braid related to $\varphi_{b,\infty}^{k(\ell,i)}$ by forgetting the strands, whose base points are not contained in $S^{\ell,i}_1\cap E'_n$, does not contain information about the geometric fat braids of generation less than $\ell$.
We will therefore relate
the irreducible braid component $\widehat{B(\ell,i)}$  to the power $b^{k(\ell,i)}$ of the braid $b$ similarly  as we related the irreducible nodal component to the power $b^{k(\ell,i)}$.  Recall that, the irreducible nodal component $\widehat{\mathfrak{m}_{b,\odot}^{\ell,i}}$ can be identified with the class $\reallywidehat{\mathfrak{m}_{b^{k(\ell,i)},\odot}\mid Y^{\ell,i}_1}$
that is related to the power $b^{k(\ell,i)}$ of the braid $b$ and a set $Y^{\ell,i}_1$ of the cycle $\mbox{cyc}^{\ell,i}$.
\index{$\widehat{\mathbold{B}(\ell,i)}$}

In detail, we
extend the family $\varphi_t,\, t \in [0,1],$ to a family defined for all $t \in \mathbb{R}$ by putting $\varphi_{t} \stackrel{def}= \varphi_{t-k} \circ (\varphi_{b,\infty})^{k} , t \in [k,k+1]$. Since  $\varphi_0={\rm Id}$ and $\varphi_1=\varphi_{b,\infty}$, the family is correctly defined. 
The family $\varphi_t,\, t \in [0,k(\ell,i)],$ is a parameterizing isotopy for $b^{k(\ell,i)}$ (normalized on the interval $[0,k(\ell,i)]$ rather than on the interval $[0,1]$).

In other words, the geometric braid
\begin{equation}
\label{eq7a.14}\Big\{\big(t, \varphi_t (E_n)\big) , \quad  t \in [0,k (\ell,i)] \Big\} \,,
\end{equation}
represents $b^{k(\ell,i)}$.
Indeed, $ \varphi_0=\rm{Id}$, and $\varphi_{k(\ell,i)}=
\varphi_{b,\infty}^{k(\ell,i)}$.

Define a tubular braid in $[0,k(\ell,i)]\times \mathbb{C}$ as follows.
Let as before $S^{\ell,i}_j, \, j=1,\ldots,k(\ell,i),$ be the sets of the cycle $\rm{cyc}_0^{\ell,i}$. 
Put $E^{\ell,i}_1\stackrel{def}= E \cap S_1^{\ell,i}$.
Further, consider all interior boundary components of $S_1^{\ell,i}$ and the closed topological discs of generation $\ell+1$ in ${\mathbb{C}}$ that are bounded by them (in other words, consider all ''interior holes'' of $S_1^{\ell,i}$). They are of the form $\overline{\delta_{j'}^{\ell+1,i'}}$, where ${\delta_{j'}^{\ell+1,i'}}$ is the bounded disc in $\mathbb{C}$ that is bounded by the exterior boundary component of $S^{\ell+1,i'}_{j'}$. Let $H({S_1^{\ell,j}} ) $ be the set 
of ''interior'' holes of $S_1^{\ell,j}$.
Consider the tubular braid
\begin{align}\label{eq7a.17}
\left\{\Big(t, \;  \varphi_t \big(E^{\ell,i}_1\, \cup\, \bigcup_ {H({S_1^{\ell,j}} ) } \overline{\delta_{j'}^{\ell+1,i'}}\big)\Big),\, t \in [0,k (\ell,i)]\right\}
\end{align}
in the tube $[0,k(\ell,i)]\times \mathbb{C}$.

\begin{figure}[h]
\begin{center}
\includegraphics[width=120mm]{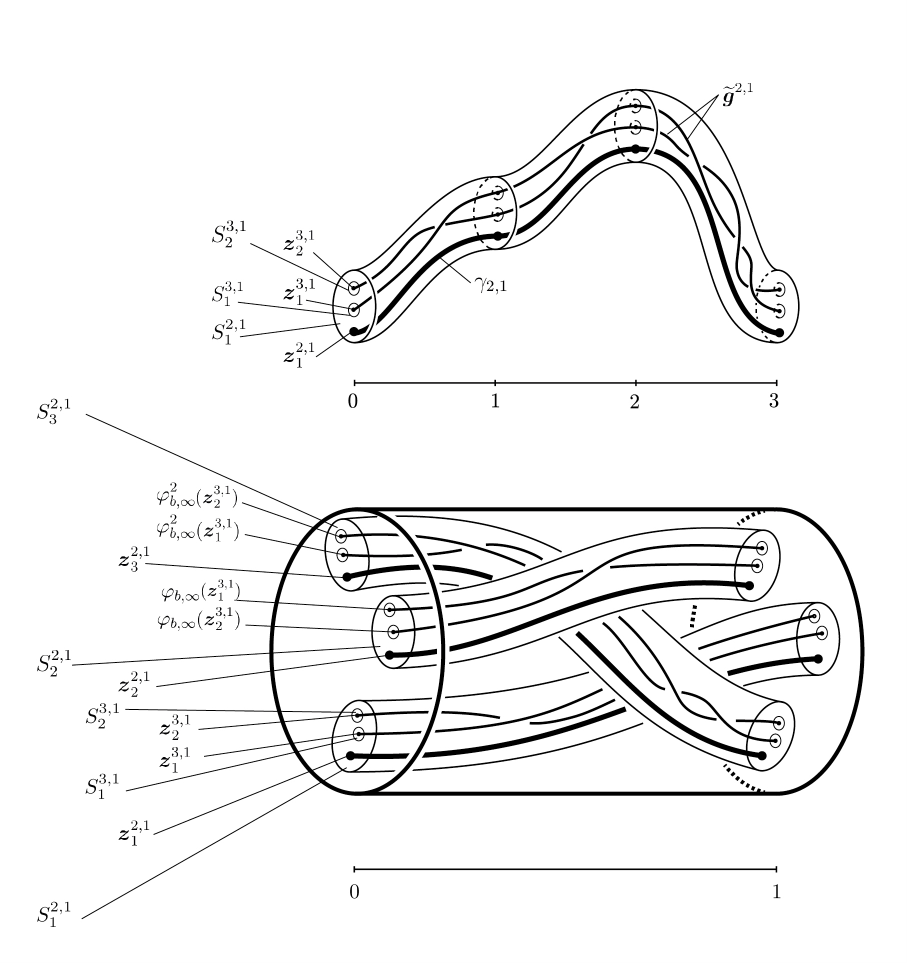}
\end{center}
\caption{The irreducible braid component $\widehat{\mathbold{B}({2,1})}$ of a non-pure braid $b$.}\label{fig7a.2}
\end{figure}

For each ''interior hole'' $ \overline{\delta_{j'}^{\ell+1,i'}}\in H(S_1^{\ell,j})$ of
$S_1^{\ell,j}$ we consider the point $\mathbold{z}_{j'}^{\ell+1,i'} \in  S_{j'}^{\ell+1,i'} \subset    \overline{\delta_{j'}^{\ell+1,i'}}$ associated to 
$\overline{\delta_{j'}^{\ell+1,i'}}$
by Lemma \ref{lem7a.1}. Let $\mathbold{E}^{\ell,i}_1$ be the set of all such points $\mathbold{z}_{j'}^{\ell+1,i'} $.
The geometric fat braid which is a deformation retract of \eqref{eq7a.17}
is defined as
\begin{equation}
\label{eq7a.18} \Big\{\big(t,\varphi_t (E^{\ell,i}_1 \cup \mathbold{ E}^{\ell,i}_1)\big) ,
\;\;  t \in [0,k (\ell,i)]\Big\}
\end{equation}
(See Figure 6.2 where $\mathbold{ E}^{2,1}_1 = \{ \mathbold{z}_1^{3,1} , \mathbold{z}_2^{3,2}
\}$, $E^{2,1}_1 = \emptyset$.)

Denote by $n(\ell,i)$ the number of points in $E^{\ell,i}_1 \cup
{\mathbold E}^{\ell,i}_1$. We write the fat mapping defining the geometric fat braid \eqref{eq7a.18} as
\begin{equation}
\label{eq7a.20} \widetilde{\mathbold{ g} ^{\ell,i}}^{k(\ell,i)} : [0,k(\ell,i)] \to
C_{n(\ell,i)} ({\mathbb C}) \diagup {\mathcal S}_{n(\ell,i)} \, .
\end{equation}
\index{$\tilde {\mathbold{g}}^{\ell,i}$}
\noindent Reparameterize the geometric fat braid $\widetilde{\mathbold{ g}^{\ell,i}}^{k(\ell,i)}$ to obtain a
mapping defined on the unit interval, in other words, put
\begin{equation}\label{eq7a.20'a}
\mathbold{g}^{\ell,i} (t) = \widetilde{\mathbold{ g}^{\ell,i}}^{k(\ell,i)} \left(
k(\ell,i)\,t \right), \;\; \mathbold{g}^{\ell,i}   : [0,1] \to C_{n(\ell,i)} ({\mathbb C})
\diagup {\mathcal S}_{n(\ell,i)} \, .
\end{equation}

The isotopy class of the geometric fat braid $\mathbold{g}^{\ell,i}$ is denoted by
$\mathbold{B}(\ell,i)$. \index{$\mathbold{B}(\ell,i)$} Its conjugacy class
$\widehat{\mathbold{B}(\ell,i)}$ is the irreducible braid component of the braid $b$ corresponding to the cycle $\rm{cyc}^{\ell,i}$.
The conjugacy class $\widehat{\mathbold{B}(\ell,i)}$
can be identified with the free isotopy class of $\mathbold{g}^{\ell,i}$. It does not depend on the choice of the set $S^{\ell,i}_1$ of the cycle $\rm{cyc}^{\ell,i}$.
\index{$\widehat{\mathbold{B}(\ell,i)}$}

The equality
\begin{equation}\label{eq7a.20''}
\reallywidehat{\mathfrak{m}_{b^{k(\ell,i)},\odot}\mid Y^{\ell,i}_1}=\reallywidehat{\mathfrak{m}_{B(\ell,i),\infty}}
\end{equation}
follows by applying the equation \eqref{eq6.31} to the braid obtained from $b^{k(\ell,i)}$ by omitting all strands, that are not contained in $\delta_1^{\ell,i}$.

\section{The building block of the recovery. 
The general case.}\label{sec:7a.3b}
In the case of non-pure braids it is more transparent to work with the associated closed braids 
rather than with the braids themselves. We obtain a closed geometric fat braid from a geometric fat braid
\begin{equation}\label{eq7a.10}
\Big\{(t,\mathbold{g}(t)), \, t \in [0,1]\Big\}
\end{equation}
in the cylinder $[0,1]\times\mathbb{C}$ by replacing the interval $[0,1]$ by the quotient $[0,1]\diagup 0\sim1$ which is diffeomorphic to the unit circle $\partial \mathbb{D}$. In other words, the closed geometric fat braid is obtained from the geometric braid by gluing together the top and the bottom fiber of the cylinder $[0,1]\times \mathbb{C}$ using the identity mapping of $\mathbb{C}$. We put $\mathbold{G}(e^{2\pi i t})=\mathbold{g}(t),\, t \in [0,1],$ which is well defined since $\mathbold{g}(0)=\mathbold{g}(1)$.
The associated closed geometric fat braid, denoted by $\mathbold{G}$, is defined as the subset of the solid torus $\partial \mathbb{D}\times \mathbb{C}$
\begin{align}\label{eq7a.11}
\Big\{\big(e^{2\pi i t}, \,\mathbold{ g}(t)\big),\,t \in [0,1]\Big\}
=\Big\{\big (e^{2\pi i t}, \,\mathbold{G}(e^{2\pi i t})\big),\, e^{2\pi i t} \in \partial \mathbb{D}\Big\}\,.
\end{align}

The connected components of the closed geometric fat braid \eqref{eq7a.11} are simple closed curves in the solid torus $\partial \mathbb{D}\times \mathbb{C}$.
For each such loop $\Gamma$ there is a natural number $k$ such that the loop intersects each fiber
$\{e^{2\pi it}\}\times \mathbb{C}$ along $k$ points. We call $k$ the covering multiplicity of $\Gamma$.

Consider the $k$-fold covering $\widetilde{\partial \, {\mathbb D}}^{k}\times \mathbb{C} \ni (\zeta,z) \overset{\tilde{p}^k}\to (\zeta^{k},z) \in \partial \mathbb{D} \times \mathbb{C}$ of $\partial \mathbb{D}\times \mathbb{C}$. Notice that  $\widetilde{\partial \, {\mathbb D}}^{k}\times \mathbb{C}$ is holomorphically isomorphic to $\partial \mathbb{D}\times \mathbb{C}$.
We may identify $ \partial \mathbb{D} \times \mathbb{C}$ with $\Big\{\big(e^{2\pi i t}, z\big),\,t \in [0,1], z \in \mathbb{C}\Big\}$, and $\widetilde{\partial \, {\mathbb D}}^{k}\times \mathbb{C}$ with $\Big\{\big(e^{2\pi i \frac{t}{k}}, z\big),\,t \in [0,k], z \in \mathbb{C}\Big\}$. Then $\tilde{p}^k((e^{2\pi i \frac{t}{k}}, z))=(e^{2\pi i t}, z)$ for $(e^{2\pi i \frac{t}{k}}, z)\in
\widetilde{\partial \, {\mathbb D}}^{k}\times \mathbb{C}$.

A loop of covering multiplicity $k$ can be written as
\begin{align}\label{eq7a.6''}
{\Gamma}=\Big\{\big(e^{2\pi it}, \, \mathbold{\gamma}(t)\big),\; t \in [0,k]\Big\}\,
\end{align}
for a function $\gamma$ with $\gamma(0)=\gamma(k)$.
If  ${\Gamma}$ 
has covering multiplicity $k$,
the lift ${\widetilde{\Gamma}°}^k$ of $\Gamma$ to the $k$-fold covering,
\begin{align}\label{eq7a.6'''}
{\widetilde{\Gamma}°}^k= \Big\{\big(e^{2\pi i \frac{t}{k}},  \, {\gamma}( t)  \big), \,t \in [0,k], z \in \mathbb{C}\Big\}
\end{align}
 has covering multiplicity $1$ as subset of $\widetilde{\partial \, {\mathbb D}}^{k}\times \mathbb{C}$. 
\index{$\widetilde{\partial \, {\mathbb D}}^{k}\times \mathbb{C}$}

\smallskip
A solid torus $T$ in $\partial \mathbb{D}\times \mathbb{C}$ that intersects each fiber  $\{e^{2\pi it}\}\times \mathbb{C}$ along a set of $k$ disjoint closed discs is called a solid torus of covering multiplicity $k$. A solid torus $T$ of covering multiplicity $k$ can be written as $\cup_{t\in[0,k]} \,\{e^{2\pi it}\}\times U_t$, where for each $t$ the sets $U_t,\,\ldots,\,U_{t+k-1}$ are disjoint topological discs in $\mathbb{C}$.
Take the lift  ${{\widetilde T}°}^k= \cup_{t\in[0,k]} \,\{e^{2\pi i \frac{t}{k}}\}\times U_t$ under $\tilde{p}^k$ of the solid torus $T$ of covering multiplicity $k$. 
The torus ${\widetilde {T}°}^k$ has covering multiplicity $1$ as subset of the $k$-fold covering
space $\widetilde{\partial \, {\mathbb D}}^{k}\times \mathbb{C}
\cong \partial \mathbb{D}\times \mathbb{C}$.

The following procedure is the basis for the recovery of conjugacy classes of not necessarily pure reducible braids.
Let $\mathbold G:\partial \mathbb{D}\to C_n(\mathbb{C})\diagup \mathcal{S}_n$ be a closed geometric fat  braid in $\partial \mathbb{D}\times \mathbb{C}$, 
written as 
$\Big\{\big(e^{2\pi it}, \, \mathbold g(t)\big),\, t \in [0,1]\Big\}\,$ 
with $\mathbold{G}(e^{2\pi it})=\mathbold{g}(t),\,t\in [0,1]$ (see equation \eqref{eq7a.11}).
Let $\mathbold{\Gamma}$ be a fat loop of covering multiplicity $k$ which is a part of the closed geometric fat braid \eqref{eq7a.11},
\begin{align}\label{eq7a.8}
\mathbold{\Gamma}=\Big\{\big(e^{2\pi it}, \, \mathbold{\gamma}(t)\big),\; t \in [0,k]\Big\}\,
\end{align} 
for a fat function $\mathbold\gamma$ with $\mathbold\gamma(0)=\mathbold\gamma(k)$. 
The lift ${\widetilde{\mathbold{\Gamma}}°}^k$ can be written as
\begin{align}\label{eq7a.6**}
{\widetilde{\mathbold{\Gamma}}°}^k=\big\{(e^{2\pi i \frac{t}{k}},\mathbold{\gamma}(t)): t\in [0,k]\big\}\,.
\end{align}
${\widetilde{\mathbold{\Gamma}}°}^k$ is the graph of a function depending on $e^{2\pi i \frac{t}{k}}\in \widetilde{\partial \mathbb{D}}^k$.

Take a solid torus $T$ in $\partial \mathbb{D} \times \mathbb{C}$ of covering multiplicity $k$, that surrounds the fat loop $\mathbold{\Gamma}$ of covering multiplicity $k$ (more precisely, $\mathbold{\Gamma}$ is a deformation retract of $T$),
such that $T$ does not intersect any other loop of the closed geometric fat braid $\mathbold{G}$. 
Replacing $\mathbold{\Gamma}$ by $T$ we obtain a closed geometric tubular fat braid. (A closed geometric tubular fat braid is defined in a similar way as a geometric tubular fat braid.)

Let ${\widetilde{\mathbold{F}}°}^k$ be a closed geometric fat braid in $\widetilde{\partial \, {\mathbb D}}^{k}\times \mathbb{C}$ that is contained in the lift ${\widetilde{T}°}^k$. 
We write
\begin{align}\label{eq7.a19'}
{\widetilde{\mathbold{F}}°}^k=\big\{(e^{2\pi i \frac{t}{k}},{\widetilde{\mathbold{f}}°}^k(t)): t\in[0,k]\big\}
\end{align}
with ${\widetilde{\mathbold{f}}°}^k(t)\stackrel{def}={\widetilde{\mathbold{F}}°}^k(e^{2\pi it}),\, t \in [0,k]\,,$ ${\widetilde{\mathbold{f}}°}^k:[0,k]\to  C_n(\mathbb{C})\diagup \mathcal{S}_n$ for some natural number $n$. 
We define the following operation that replaces the fat loop $\mathbold{\Gamma}$ of the closed fat braid $\mathbold{G}$ by ''a copy of the closed fat braid $\mathbold{F}=\tilde{p}^k({\widetilde{\mathbold{F}}°}^k)=\big\{(e^{2\pi i t},{\widetilde{\mathbold{f}}°}^k(t)): t\in[0,k]\big\}$ inserted into $T\,$'':
\begin{align}\label{eq7a.13}
\mathbold{G}\sqcup_{\mathbold{\Gamma}}{\widetilde{\mathbold{F}}°}^k\stackrel{def}= \Big\{\big(e^{2\pi it},\, \mathbold{g}(t)\sqcup_{\mathbold{\gamma}(t)}{\widetilde {\mathbold{f}}°}^k(t)\big),\, t \in [0,k]\Big\}\,.
\end{align}
The operation $\mathbold{g}(t)\sqcup_{\mathbold{\gamma}(t)} \tilde{\mathbold{f}}^k(t) \stackrel{def} =
\big(\mathbold{g}(t)\setminus \{\mathbold{\gamma}(t)\}\big) \cup \tilde{\mathbold{f}}^k(t)$ for every $t\in [0,k]$ is defined as in Chapter \ref{chapter6} equation
\eqref{eq6.6}, namely, for each $t$ we replace the point $\mathbold{\gamma}(t)$
of the unordered tuple of points $\mathbold{g}(t)$ by the unordered tuple of points ${\widetilde{\mathbold{f}}°}^k(t)$.
\index{$\mathbold{G}\sqcup_{\mathbold{\Gamma}}\mathbold{F}$}

Since the closed geometric fat braid defined by ${\widetilde{\mathbold{F}}°}^k$ is contained in ${\widetilde{ T}°}^k$, the closed geometric fat braid $\Big\{\big(e^{2\pi it},{\widetilde{\mathbold{f}}°}^k(t)\big), \,t \in [0,k]\Big\}$ is contained in $T$. Hence, \eqref{eq7a.13} defines a closed geometric fat braid.

We need the following analog of Lemma \ref{lem6.1''}.

\begin{lemm}\label{lem7a.2}
Let $\mathbold{G}^0\,$ and $\mathbold{G}^1\,$ 
be (not necessarily pure) closed geometric fat braids in $\partial \mathbb{D}\times \mathbb{C}$, such that $\mathbold{G}^0$ is free isotopic to $\mathbold{G}^1$. 
Let $\mathbold{\Gamma}_0$ be a fat loop of $\mathbold{G}^0$  of covering multiplicity $k$, and let $\mathbold{\Gamma}_1$ be the fat loop  of $\mathbold{G}^1$  of covering multiplicity $k$ that corresponds to $\mathbold{\Gamma}_0$ under the isotopy joining $\mathbold{G}^0$ and $\mathbold{G}^1$.

For $j=0,1$ we let ${T}(\mathbold{\Gamma}_j)$ be a tube around $\mathbold{\Gamma}_j$ which does not meet the other loops of $\mathbold{G}^j\,$
and has  $\mathbold{\Gamma}_j$ as a deformation retract.
Suppose for $j=0,1$ there are closed geometric fat braids $\widetilde{(\mathbold{F}^j)}^k$ in $\widetilde{\partial \, {\mathbb D}}^{k}\times \mathbb{C}$ that are contained in the lift $\widetilde {T(\mathbold{\Gamma}_j)}^k$ of $T(\mathbold{\Gamma}_j)$.

If in addition $\;\widetilde{(\mathbold{F}^0)}^k\;$ and  $\;\widetilde{(\mathbold{F}^1)}^k\;$ are free isotopic in  $\widetilde{\partial \, {\mathbb D}}^{k}\times \mathbb{C}$,
then  $\;\mathbold{G}^0 \sqcup_{\mathbold{\Gamma}_0} \widetilde{(\mathbold{F}^0)}^k\;$ and $\;\mathbold{G}^1 \sqcup_{\mathbold{\Gamma}_1} \widetilde{(\mathbold{F}^1)}^k\;$ are free isotopic closed geometric fat braids.
\end{lemm}

\noindent {\bf Proof.} 
Let $\mathbold{G}^s,\, s \in [0,1],$ be an isotopy of closed geometric fat braids joining $\mathbold{G}^1$ with $\mathbold{G}^2$,
and let $\mathbold{\Gamma}_s$ be the family of loops in $\mathbold{G}^s$ that is obtained by this isotopy and joins $\mathbold{\Gamma}_0$ with $\mathbold{\Gamma}_1$. Choose $\varepsilon$
so small that for each $s\in [0,1]$ the $\varepsilon$-neighbourhood in $\partial \mathbb{D} \times \mathbb{C}$ of the closed geometric fat braid $\Big\{\big(e^{2\pi it}, \mathbold{G}^s(e^{2\pi it})\big), \, t \in[0,1]\Big\}$ is a closed geometric tubular braid.

For $j=0,1$ the closed geometric fat braid  $\;\widetilde{(\mathbold{F}^j)}^k\;=\Big\{\big(e^{2\pi i \frac{t}{k}},\,\widetilde{(\mathbold{f}^j)}^k(t)\big), \, t \in[0,k],\Big\}$ in $\widetilde{\partial\,{\mathbb D}}^{k}\times \mathbb{C}$ is contained in $\widetilde{T(\mathbold{\Gamma_j})}^k$. 
By the proof of Lemma \ref{lem6.1'} for $j=0,1$ there exists an isotopy of closed geometric
fat braids $\widetilde{(\mathbold{F}^j_s)}^k, \, s\in [0,1],$
such that for $j=0,1$, each  $\widetilde{(\mathbold{F}^j_s)}^k$ 
is contained in $\widetilde{T(\mathbold{\Gamma}_j)}^k$, and for each $j$ the family
joins the closed geometric fat braid $\widetilde{(\mathbold{F}^j)}^k$ with a closed geometric fat braid  $\widetilde{(\mathbold{F}^j_1)}^k$ contained in the $\varepsilon$-neighbourhood ${T_{\varepsilon}(\widetilde{\mathbold{\Gamma}_j}^k)}$ of the lift $\widetilde{\mathbold{\Gamma}_j}^k$ of $\mathbold{\Gamma}_j$ to the $k$-fold covering of $\partial\mathbb{D}\times \mathbb{C}$.

Since $\mathbold{G}^j \sqcup_{\mathbold{\Gamma}_j} \widetilde{(\mathbold{F}^j_s)}^k,\, s \in [0,1]$, is a free isotopy of closed geometric fat braids for $j=0,1$, we may assume from the beginning, that $\widetilde{\mathbold{F}^j}^k$ is contained in the $\varepsilon$-tube $T_{\varepsilon}(\widetilde{\mathbold{\Gamma}_j}^k)$ around $\widetilde {\mathbold{\Gamma}_j}^k$.
Then $\widetilde{\mathbold{F}^j}^k,\, j=0,1,$ has the form  $\widetilde{\mathbold{\Gamma}_j}^k\boxplus \varepsilon \widetilde{\mathring{\mathbold{F}}^j }^k\stackrel{def}= \big\{(e^{2\pi i\frac{t}{k}}, \gamma_j(t)\boxplus \varepsilon \widetilde{\mathbold{\mathring{f}}_j}^k(t)),\, t\in[0,k]\big\}$ for a closed geometric fat braid $\widetilde{\mathring{\mathbold{F}}_j}^k$, that is defined by $\big\{(e^{2\pi i\frac{t}{k}},\widetilde{\mathbold{\mathring{f}}_j}^k(t)),\, t\in[0,k]\big\}$ and contained in $\widetilde{\partial\, \mathbb{D}}^k\times \mathbb{D}$.
Recall that for each positive number $\varepsilon$ the set $\varepsilon \widetilde{ \mathring{\mathbold{f}}_j}^k(t)$ is obtained from $\widetilde{ \mathring{\mathbold{f}}_j}^k(t)$ (considered as subset of $\mathbb{C}$) by multiplying each point of the latter set by  $\varepsilon$, and for each
$t\in [0,k]$ the set
$\gamma(t)\boxplus \varepsilon \widetilde{ \mathring{\mathbold{f}}_j}^k(t)$ is obtained by adding $\gamma(t)$ to each point of the set $\varepsilon \widetilde{\mathring{\mathbold{f}}_j}^k(t)$.
Since the closed geometric fat braids  $\widetilde{\mathbold{\Gamma}_0}^k\boxplus \varepsilon \widetilde{\mathring{\mathbold{F}}^0}^k$ and  $\widetilde{\mathbold{\Gamma}_1}^k\boxplus \varepsilon \widetilde{\mathring{\mathbold{F}}^1}^k$ are free isotopic, also the closed geometric fat braids
$\widetilde{\mathring{\mathbold{F}}_0}^k$ and $\widetilde{\mathring{\mathbold{F}}_1}^k$ are free isotopic. 
Let $\widetilde{\mathring{\mathbold{F}}^s}^k,\, s\in[0,1]$, be an isotopy of closed geometric fat braids contained in $\widetilde{\partial\, \mathbb{D}}^k\times \mathbb{D}$ and joining  $\widetilde{\mathring{\mathbold{F}}^0}^k$ and  $\widetilde{\mathring{\mathbold{F}}^1}^k$.

By the choice of $\varepsilon\,$ the family $\mathbold{G}^s\sqcup_{\mathbold{\Gamma}_s}
\widetilde{\mathbold{\Gamma}_s}^k \boxplus \varepsilon\widetilde{\mathring{\mathbold{F}}^s}^k$ 
is a free isotopy of closed geometric fat braids joining $\mathbold{G}^0\sqcup_{{\mathbold{\Gamma}_0}}
\widetilde{\mathbold{\Gamma}_0}^k \boxplus
 \varepsilon\widetilde{\mathring{\mathbold{F}}^0}^k$ and $\mathbold{G}^1\sqcup_{\mathbold{\Gamma}_1} \widetilde{\mathbold{\Gamma}_1}^k \boxplus
\varepsilon\widetilde{\mathring{\mathbold{F}}^1}^k$. We assumed that the latter two closed geometric fat braids are equal to  $\mathbold{G}^0 \sqcup_{\mathbold{\Gamma}_0} \widetilde{(\mathbold{F}^0)}^k$ and $\mathbold{G}^1 \sqcup_{\mathbold{\Gamma}_1} \widetilde{(\mathbold{F}^1)}^k$, respectively.
The Lemma is proved. \hfill $\Box$

\medskip

Notice that for any closed geometric fat braid given by a mapping $\mathbold{F}$ and any diffeomorphism $\sigma:\partial \mathbb{D}\to \partial \mathbb{D}$ the two closed fat braids
$$
\Big\{\big(e^{2\pi it},\mathbold{F}(e^{2\pi it })\big), \, t\in [0,1]\Big\}\;\mbox{ and} \;\Big\{\big(e^{2\pi it},\mathbold{F}(e^{2\pi i \sigma(t) })\big), \, t\in [0,1]\Big\}
$$
are free isotopic.

\section [Recovery of the conjugacy class of braids. The general case.]{Recovery of conjugacy classes of braids from the irreducible braid components. The general case.}
\label{sec:7a.3}
The plan is the following.
We consider a geometric braid $g$,
$g:[0,1]\to C_n(\mathbb{C})\diagup \mathcal{S}_n$
with base point $E_n$ that represents a braid $b\in\mathcal{B}_n$, 
and let  $\;{\mathbold{g}}^{1,1}:[0,1]\to C_{n(1,1) }\diagup \mathcal{S}_{n(1,1) }      \;,\;$  $\;\widetilde{\mathbold{g}^{\ell,i}}^{k(\ell,i)}: [0,k(\ell,i)]\to C_{n(\ell,i) }\diagup \mathcal{S}_{n(\ell,i) }\;$, and $\;\;{\mathbold{g}^{\ell,i}}
(t)=\widetilde{\mathbold{g}^{\ell,i}}^{k(\ell,i)}(k(\ell,i) t), \, t\in [0,1]$,  be the mappings constructed in Section \ref{sec:7a.2}. For uniformity of notation we put $\widetilde{\mathbold{g}^{1,1}}^1\stackrel{def}={\mathbold{g}}^{1,1}$. We write  $\mathbold{G}^{1,1}(e^{2\pi it}) \stackrel{def}=g^{1,1}(t),\,t\in [0,1]     $, 
${\mathbold{G}^{\ell,i}}(e^{2\pi it})={\mathbold{g}^{\ell,i}}(t),\, t\in [0,1], $, and
$\widetilde{\mathbold{G}^{\ell,i}}^{k(\ell,i)}(e^{2\pi i\frac{t}{k(\ell,i)}})\stackrel{def}=\widetilde{\mathbold{g}^{\ell,i}}^{k(\ell,i)}(t),\; t\in [0,k(\ell,i)]$, for all $\ell$ and  $i$. 
We will write the closed geometric braid $G$, that is associated to $g$, 
in terms of the $\widetilde{\mathbold{G}^{\ell,i}}^{k(\ell,i)}$.

Each irreducible braid component $\widehat{B(\ell,i)}$ will be represented by a suitable geometric fat braid $\mathbold{f}^{\ell,i}$. The associated closed braids $\widetilde{\mathbold{F}^{\ell,i}}^{k(\ell,i)}$ are put together in the same way as the $\widetilde{\mathbold{G}^{\ell,i}}^{k(\ell,i)}$ were put together. Lemma \ref{lem7a.2} will imply that this procedure gives a closed braid that represents $\hat b$.

As in Section \ref{sec:7a.2} we consider a parameterizing isotopy $\varphi_t$ for the geometric braid $\{g(t): \,t\in [0,1]\}$ representing $b$, i.e. a continuous family $\varphi_t\in {\rm Hom}(\mathbb{P}^1;\mathbb{P}^1\setminus\mathbb{D}^1),\, t\in[0,1],$ with $\varphi(0)={\rm Id}$ and $\varphi_1=\varphi_{b,\infty}\in \mathfrak{m}_{b,\infty}$, such that $g(t)=\varphi_t(E_n),\,t\in[0,1]$.
We assume as in Section \ref{sec:7a.2} that $\varphi_{b,\infty}$ fixes setwise a
system of admissible curves $\mathcal{C}$ that completely reduces the mapping class of $\varphi_{b,\infty}$. We also assume that $\varphi_{b,\infty}$ satisfies the conditions for the mapping $\varphi'_{b,\infty}$ of Lemma \ref{lem7a.1}.
The geometric tubular braid \eqref{eq7a.2} that was associated to the component $S^{1,1}$ of $\mathbb{P}^1\setminus \bigcup_{C \in \mathcal{C}} C$ and to the parameterizing isotopy $\varphi_t$ defines a closed  geometric tubular braid
\begin{equation}\label{eq7a.3}
\left\{\Big( e^{2\pi i t},\,
\varphi_t (E^{1,1} \cup \bigcup_{i=1}^{k_2}  \,\bigcup_{j=1}^{k(2,i)}\,
\overline{\delta_j^{2,i}}) \Big)\, ,t \in [0,1]\right\}.
\end{equation}
The connected components of the ordinary part $\Big\{\big(e^{2\pi i t}, \varphi_t (E^{1,1})\big) ,\, t\in [0,1]\Big\}$ of the closed tubular braid  \eqref{eq7a.3} are simple closed loops. The loops 
correspond to the $\varphi_{b,\infty}$-cycles of points in $E^{1,1}$. They are obtained by gluing together the
ordinary strands of \eqref{eq7a.2} that have initial points at points of $E^{1,1}$.

The union of the tubular strands appearing in \eqref{eq7a.3} is a union of closed solid tori $T^{2,i}$ in $\partial \mathbb{D} \times {\mathbb{C}}$. 
The  $T^{2,i}$ correspond to the $\varphi_{b,\infty}$-cycles ${\rm cyc}_0^{2,i}$, $i=1, \ldots, k_2$. Each $T^{2,i}$
intersects each fiber $\{z\} \times {\mathbb{C}}$ along $k(2,i)$ topological discs. The intersection of $T^{2,i}$ with the fiber
$\{1\}\times {\mathbb{C}}$ is the union of the closed  topological discs $\{1\} \times \overline{\delta_j^{2,i}}\,, j=1,\ldots,k(2,i)$. 
The torus $T^{2,i}$ is obtained by ''gluing together''
the $k(2,i)$ tubes $\bigcup_{t \in [0,1]}\{t\} \times \varphi_t(\delta_j^{2,i})$, $j=1, \ldots,k(2,i)$. (We used the equality $\varphi_t(\delta_j^{2,i})=\varphi_{t+j-1}(\delta_1^{2,i})$ for $t\in[0,1]$ and
$j=1, \ldots,k(2,i)$.)

With the $\mathbold{z}_j^{2,i}\in \delta_j^{2,i} $ from Lemma \ref{lem7a.1} we associate to the
geometric fat braid  \eqref{eq7a.4''} (see also \eqref{eq7a.5'})
the closed geometric fat braid
\begin{equation}\label{eq7a.6'}
\biggl\{\Big( e^{2\pi i t},
\varphi_t (E^{1,1} \cup \bigcup_{i=1}^{k_2}  \,\bigcup_{j=1}^{k(2,i)}\,
\mathbold{z}_j^{2,i}) \Big)\,,\,t \in[0,1]\biggl\}\,,
\end{equation}
which is a  subset of the torus $\partial  {\mathbb D}  \times \mathbb{C}$ and is a deformation retract of the closed geometric tubular braid \eqref{eq7a.3}.
We denote the mapping that defines the closed geometric fat braid by
\begin{equation}
\label{eq7a.6}\widetilde{\mathbold{G}^{1,1}}^1
: \partial  {\mathbb D} \to C_{n(1,1)} ({\mathbb C})
\diagup {\mathcal S}_{n(1,1)}, \quad
\end{equation}
Note that $\widetilde{  \mathbold{G}^{1,1}}^1 (e^{2\pi i t}) =\widetilde{ \mathbold{g}^{1,1}}^1 (t)=\mathbold{g}^{1,1} \, ,  t \in [0,1]$ ( see  \eqref{eq7a.4''} and \eqref{eq7a.5'}).

The connected components of the closed fat braid \eqref{eq7a.6'}
are simple closed loops in the torus $\partial \mathbb{D} \times {\mathbb{C}}$, some of them are ordinary loops, some of them are fat loops.
To each cycle ${\rm cyc}_0^{2,i}$ corresponds a fat loop denoted by 
$\mathbold{\Gamma}_{2,i}\stackrel{def}=\Big\{\big(e^{2\pi i t}, \, \varphi_t(\mathbold{E}^{2,i}_1)\big),\, t\in [0,k(2,i)]\Big\} \subset \partial {\mathbb D} \times \mathbb{C}$. 
We also write this fat loop as
$\mathbold{\Gamma}_{2,i}=\Big\{\big(e^{2\pi i t}, \, \mathbold{\gamma}_{2,i}(t)\big),\, t\in [0,k(2,i)]\Big\}$. 
Since by the definition of the extension of $\varphi_t$ to $\mathbb{R}$ the equalities $\varphi_{k(\ell,i)}(\mathbold{z}^{2,1}_1)= \varphi_{b,\infty}^{k(\ell,i)}(\mathbold{z}^{2,1}_1)
=\mathbold{z}^{2,1}_1$, and $\varphi_{j}(\mathbold{z}^{2,1}_1)= \varphi_{b,\infty}^{j}(\mathbold{z}^{2,1}_1)= \mathbold{z}^{2,1}_{j+1}$ hold for $j=1,\ldots,k(\ell,i)-1$, the fat loop $\mathbold{\Gamma}_{2,i}$
intersects 
the fiber $\{1\} \times {\mathbb{C}}$ at the points $(1,\mathbold{z}_1^{2,i}), \ldots,(1,\mathbold{z}_{k(\ell,i)}^{2,i})$. It intersects each fiber $\{e^{2\pi i t}\} \times {\mathbb{C}}$ along $k(2,i)$
points.

We consider for each $(2,i)$ the closed geometric fat braid $\;\widetilde{\mathbold{G}^{2,i}}^{k(2,i)}\;$ in
 $\widetilde{\partial \mathbb{D}}^{k(2,i)} \times {\mathbb{C}}$, that is defined as
\begin{equation}\label{eq7a.12'}
\Big\{\big(e^{2\pi i\frac{t}{k(2,i)}}, \widetilde{\mathbold{g}^{2,i}}^{k(2,i)}(t)\big),\, t \in [0,k(2,i)]\Big\}
\end{equation}
where  $\widetilde{\mathbold{g}^{2,i}}^{k(2,i)}$ 
is defined by \eqref{eq7a.18} and \eqref{eq7a.20} with $\ell=2$.

The closed geometric fat braid
\begin{equation}\label{eq7a.12}
\widetilde{ \mathbold{G}^{1,1}}^1\sqcup_{\mathbold{\Gamma}_{2,1} }\widetilde{\mathbold{G}^{2,1}}^{k(2,1)}\sqcup \ldots\sqcup_{\mathbold{\Gamma}_{2,k_2} }\widetilde{\mathbold{G}^{2,k_2}}^{k(2,k_2)}\,
\end{equation}
is associated to the domain $Q_2=S^{1,1} \cup \bigcup_{i=1}^{k_2}\bigcup_{j=1}^{k(2,i)}
\big(S_j^{2,i} \cup {C}_j^{2,i}\big)$
(compare also with \eqref{eq6.11a''}) 
in the sense that the set of its ordinary strands intersects the fiber $\{1\}\times \mathbb{C}$ of $\partial \mathbb{D}\times \mathbb{C}$ along the set 
$\{1\}\times (E\cap Q_2)$, and the set of intersection points of the fat strands with
$\{1\}\times \mathbb{C}$ contains exactly one point in $\{1\}\times \delta_j^{3,i}$
for each hole $\delta_j^{3,i}$ of $Q_2$, and contains no other point. 
Indeed,
by \eqref{eq7a.13} the closed geometric fat braid defined by
\begin{equation}\label{eq7a.14b}
\widetilde{ \mathbold{G}^{1,1}}^1\sqcup_{\mathbold{\Gamma}_{2,1} }\widetilde{\mathbold{G}^{2,1}}^{k(2,1)}
\end{equation}
can be written as
\begin{equation}\label{eq7a.14a}
\left\{ \Big(e^{2\pi it}, \widetilde{ \mathbold{g}^{1,1}}^1(t)\sqcup _{\mathbold{\gamma}_{2,1}(t) }
\widetilde{\mathbold{g}^{2,i}}^{k(2,i)}(t)
\Big),\, t \in [0,k(2,1)]\right\},
\end{equation}
where $\widetilde{\mathbold{g}^{2,i}}^{k(2,i)}$ 
is defined by \eqref{eq7a.18}, or equivalently by \eqref{eq7a.20} with $\ell=2$.

By \eqref{eq7a.18} the closed fat braid \eqref{eq7a.14a} is equal to
\begin{align}\label{eq7a.15}
\left\{ \Big(e^{2\pi it},\varphi_t (E^{1,1}\cup \mathbold{E}^{1,1})\sqcup_{\varphi_t(\mathbold{z}^{2,1}_1)}\varphi_t(E^{2,1}_1\cup \mathbold{E}^{2,1}_1)\Big) ,
\;  t \in [0,k (2,1)] \right\}\,.
\end{align}

By the equalities
$\varphi_{t+j}=\varphi_{t}\circ \varphi_{b,\infty}^{j}$ for $t\in [0,1]$, $\varphi_{b,\infty}^j(\mathbold{E}^{1,1})= \mathbold{E}^{1,1}$,
$\varphi_{b,\infty}^j(\mathbold{z}^{2,1}_1)=\mathbold{z}^{2,1}_{j+1}$ for $j=1,\ldots,k(2,1)-1$, $\varphi_{b,\infty}^{k(2,1)}(\mathbold{z}^{2,1}_1)=\mathbold{z}^{2,1}_{1}$,
$\varphi_{b,\infty}^j(E^{2,1}_1\cup\mathbold{E}^{2,1}_1)=E^{2,1}_{j+1}\cup \mathbold{E}^{2,1}_{j+1}$ for $j=1,\ldots,k(2,1)-1,$ and $\varphi_{b,\infty}^{k(2,1)}(E^{2,1}_1\cup \mathbold{E}^{2,1}_1)=E^{2,1}_1\cup \mathbold{E}^{2,1}_{1}$,
the closed geometric fat braid \eqref{eq7a.15} is equal to
\begin{equation}\label{eq7a.15'}
\left\{ \Big(e^{2\pi it},\varphi_t \big(E^{1,1} \cup (\mathbold E^{1,1}\setminus \underset{j=1}{{\overset{k(2,1)} {\bigcup }}}
\{\mathbold{z}^{2,1}_j\}) \cup
\underset{j=1}{
{\overset{k(2,1)} {\bigcup }}}
(E^{2,1}_j \cup \mathbold{ E}^{2,1}_j)\big)\Big)\;\;  t \in [0,1]\right\}\,.
\end{equation}
Induction over the closed geometric fat braids
\begin{equation}\label{eq7a.12a'}
\widetilde{ \mathbold{G}^{1,1}}^1\sqcup_{\mathbold{\Gamma}_{2,1} }\widetilde{\mathbold{G}^{2,1}}^{k(2,1)}\sqcup \ldots\sqcup_{\mathbold{\Gamma}_{2,i} }\widetilde{\mathbold{G}^{2,i}}^{k(2,i)}
\end{equation}
for $i=1,\ldots,k_2,$ shows that the closed geometric fat braid \eqref{eq7a.12} is associated to the domain $Q_2$.

Induction over the number 
of the generation $\ell=1,\ldots,N$ shows that the closed geometric fat braid
\begin{align}\label{eq7a.12''}
\widetilde{ \mathbold{G}^{1,1}}^1&\sqcup_{\mathbold{\Gamma}_{2,1} }\,\widetilde{\mathbold{G}^{2,1}}^{k(2,1)}\;\;\,\sqcup \ldots\sqcup_{\mathbold{\Gamma}_{2,k_2} }\,\widetilde{\mathbold{G}^{2,k_2}}^{k(2,k_2)}\;\; \sqcup \nonumber\\
\ldots&\sqcup_{\mathbold{\Gamma}_{N,1} }\widetilde{\mathbold{G}^{N,1}}^{k(N,1)}\sqcup \ldots\sqcup_{\mathbold{\Gamma}_{N,k_{\ell}} }\widetilde{\mathbold{G}^{N,k_N}}^{k(N,k_N)}
\end{align}
is an ordinary closed geometric braid and is equal to $G$.

Suppose now that we know the irreducible braid components $\widehat{\mathbold{B}(\ell,i)}$ of the reducible non-pure braid $b$ with respect to the admissible system of curves $\mathcal{C}$.
Represent each conjugacy class of fat braids $\widehat{\mathbold{B}(\ell,i)}$
by a closed geometric fat braid ${\widetilde{\mathbold{F}^{\ell,i}}°}^{k(\ell,i)}$ in
the bounded solid torus $\widetilde{\partial \mathbb{D}}^{k(\ell,i)} \times \mathbb{D}$ which is the $k(\ell,i)$-fold covering space of $\partial \mathbb{D} \times \mathbb{D}$.
Here $k(\ell,i)$ is the length of the cycle of $S_j^{\ell,i}$ that defines   $\widehat{\mathbold{B}(\ell,i)}$.
We write ${\widetilde{\mathbold{F}^{\ell,i}}°}^{k(\ell,i)}$ as
$\Big\{\big(e^{2\pi i\frac{t}{k(\ell,i)}}, \widetilde{\mathbold{f}^{\ell,i}}^{k(\ell,i)}(t)\big),\, t \in [0,k(\ell,i)]\Big\}$ for a continuous mapping $\widetilde{\mathbold{f}^{\ell,i}}^{k(\ell,i)}$
from $[0,k(\ell,i)]$ into the symmetrized configuration space of suitable dimension.

Let $\varepsilon_2>0$ be so small that the $\varepsilon_2$-neighbourhood in  $\partial \mathbb{D} \times \mathbb{C}$ of the closed geometric fat braid $\mathbold{F}^{1,1}={\widetilde{\mathbold{F}^{1,1}}°}^{1}$ written as
$\Big\{\big(e^{2\pi it},\,\widetilde{ \mathbold{f}^{1,1}}^1(t)\big),\, t \in [0,1]\Big\}$ is the union of disjoint solid tori that retract to the loops of the closed geometric fat braid.
Let $\mathbold{F}_{2,i},\, i=1,\ldots, k_2,$ be the fat loops of $\mathbold{F}^{1,1}$. The covering multiplicity of $\mathbold{F}_{2,i}$ is equal to $k(2,i)$.

For each $i=1,\ldots,k_2,$ we will ''insert a suitable representative of $\widehat{\mathbold{B}(2,i)}$ into the $\varepsilon_2$-neighbourhood of
the loop $\mathbold{F}_{2,i}$''.
More precisely, we consider the closed geometric fat braid given by
\begin{equation}\label{eq7a.13'}
\widetilde{ \mathbold{F}^{1,1}}^1\sqcup_{\mathbold{F}_{2,1} }(\mathbold{F}_{2,1}\boxplus \varepsilon_2 \widetilde{\mathbold{F}^{2,1}}^{k(2,1)})\,.
\end{equation}
Equality \eqref{eq7a.13'} means the following.
Write $\mathbold{F}_{2,1}$ as $\{(e^{2\pi i t}, \mathbold{f}_{2,1}(t)), t\in [0,k(2,1)]\}$. Then $\mathbold{F}_{2,1}\boxplus \varepsilon_2 \widetilde{\mathbold{F}^{2,1}}^{k(2,1)}$ is the closed geometric fat braid defined by
\begin{equation}\nonumber
\Big\{ \big(e^{2\pi it}, \mathbold{f}_{2,1}(t)\boxplus  \varepsilon_2\,
\widetilde{\mathbold{f}^{2,1}}^{k(2,1)}(t)\big),\, t \in [0,k(2,1)]\Big\}\,.
\end{equation}
The closed geometric fat braid \eqref{eq7a.13'} is defined by
\begin{equation}\label{eq7a.14''}
\Big\{ \big(e^{2\pi it}, \widetilde{\mathbold{f}^{1,1}}^1(t)\sqcup _{\mathbold{f}_{2,1}(t) }\mathbold{f}_{2,1}(t)\boxplus  \varepsilon_2\, \widetilde{\mathbold{f}^{2,1}}^{k(2,1)}(t)\big),\, t \in [0,k(2,1)]\Big\}\,.
\end{equation}
By the choice of $\varepsilon_2$  Lemma \ref{lem7a.2} applies to the pair of closed geometric fat braids ($\widetilde{\mathbold{G}^{1,1}}^1$, $\widetilde{\mathbold{F}^{1,1}}^1$) with chosen
 pair of fat loops  ($\mathbold{\Gamma}_{2,1}$, $\mathbold{F}_{2,1}$), and the pair of closed geometric  fat braids  ($\widetilde{\mathbold{G}^{2,1}}^{k(2,1)}$,   $\mathbold{F}_{2,1}\boxplus \varepsilon_2 \widetilde{\mathbold{F}^{2,1}}^{k(2,1)}$).
Hence, the closed geometric fat braid \eqref{eq7a.13'} is free isotopic to the closed geometric fat braid \eqref{eq7a.14b} 
or equivalently to  \eqref{eq7a.14a}.

By induction on the cycles of generation $2$ the closed geometric fat braid
\begin{equation}\label{eq7a.16}
\widetilde{\mathbold{F}^{1,1}}^1\sqcup_{\mathbold{F}_{2,1} }(\mathbold{F}_{2,1}\boxplus \varepsilon_2  \widetilde{\mathbold{F}^{2,1}}^{k(2,1)})\sqcup \ldots\sqcup_{\mathbold{F}_{2,k_2} }(\mathbold{F}_{2,k_2}\boxplus \varepsilon_2  \widetilde{\mathbold{F}^{2,k_2}}^{k(2,k_2)})
\end{equation}
is free isotopic to  \eqref{eq7a.12}.

By induction over the number of the 
generation $\ell=1,\ldots,N,$ 
we find
small enough numbers $\varepsilon_{\ell+1}$ and put suitable representatives
$\varepsilon_{\ell+1}\widetilde{\mathbold{F}^{\ell+1,i' }}^{\ell+1,i'}$
of $\reallywidehat{\mathbold{B}(\ell+1,i')}$ into the $\varepsilon_{\ell}$-neighbourhoods of
the fat loops $\mathbold{F}_{ \ell+1,i'}$ that correspond to the cycles of holes of $Q_{\ell}$.
For $\ell=N$ we arrive at
the closed geometric fat braid
\begin{align}\label{eq7a.60}
\widetilde{\mathbold{F}^{1,1}}^1\sqcup_{\mathbold{F}_{2,1} }(\mathbold{F}_{2,1}\boxplus \varepsilon_2  \widetilde{\mathbold{F}^{2,1}}^{k(2,1)})\sqcup \ldots\sqcup_{\mathbold{F}_{2,k_2} }(\mathbold{F}_{2,k_2}\boxplus \varepsilon_2  \widetilde{\mathbold{F}^{2,k_2}}^{k(2,k_2)}) \nonumber\\
\sqcup \ldots\sqcup_{\mathbold{F}_{N,1}}(\mathbold{F}_{N,1}\boxplus \varepsilon_N \widetilde{\mathbold{F}^{N,1}}^{k(N,1)})
\sqcup \ldots \sqcup_{\mathbold{F}_{N,k_N}}(\mathbold{F}_{N,k_N}\boxplus
\varepsilon_N \widetilde{\mathbold{F}^{N,1}}^{k(N,k_N)})\,.
\end{align}
An inductive application of Lemma \ref{lem7a.2} shows, that the closed geometric fat braid \eqref{eq7a.60} is an ordinary closed braid and is free isotopic to the closed braid \eqref{eq7a.12''} which is equal to $G$.
The recovery procedure is described.

\section[Proof of  Main Theorem for reducible braids. General
case] {Proof of  the Main Theorem for reducible braids. The general
case}\label{sec:7a.4}
The following lemma describes the entropy of the nodal conjugacy class $\reallywidehat{{\mathfrak m}_{b,\odot}}$, that is associated to $\mathfrak{m}_{b,\infty}$ and an admissible system of curves $\mathcal{C}\subset \mathbb{D}$ that completely reduces $\mathfrak{m}_{b,\infty}$,  in terms of the entropy of irreducible nodal components.

Recall that we identify mapping classes on (possibly not connetced) punctured Riemann surfaces (or on punctured nodal surfaces) with mapping classes on (possibly not connected) closed Riemann surfaces (or on closed nodal surfaces) with distinguished points. If for notational convenience we write $h(\mathfrak{m})$ for a mapping class $\mathfrak{m}$ on a punctured surface, we mean the entropy of the respective mapping class on the closed surface with distinguished points. Similarly, we call a self-homeomorphism of a
punctured Riemann surface non-periodic absolutely extremal if its extension to the closed Riemann surface with the respective set of  distinguished points  is so.

\begin{lemm}\label{lemm7.3a}
\begin{align}
h \left( \reallywidehat{{\mathfrak m}_{b,\odot}} \right)=
\max_{{\rm cyc}^{\ell,i}_\odot} \,
h \left( \reallywidehat{{\mathfrak m}_{b,\odot} \mid
{{\rm cyc}^{\ell,i}_{\odot}}} \right) =
\max_{{\rm cyc}^{\ell,i}_\odot} \,
\frac1{k(\ell,i)} \, h \left(\reallywidehat{{\mathfrak
m}_{b^{k(\ell,i)},\odot} \mid {Y_1^{\ell,i}}} \right) \, .
\end{align}
\end{lemm}
\noindent {\bf Proof.} The first equality is an easy consequence of Theorem~4 of \cite{AKM}.

To obtain the second equality it is enough to prove that the equation
\begin{align}\label{7a.20'}
h \left( \reallywidehat{{\mathfrak m}_{b,\odot} \mid
{{\rm cyc}^{\ell,i}_{\odot}}} \right) =
\frac1{k(\ell,i)} \, h \left(\reallywidehat{{\mathfrak
m}_{b^{k(\ell,i)},\odot} \mid {Y_1^{\ell,i}}} \right)\,
\end{align}
holds for each nodal cycle ${\rm cyc}^{\ell,i}_{\odot}$. 
For each mapping $\varphi_{\odot}$ representing the class
$\reallywidehat{{\mathfrak m}_{b,\odot} \mid {{\rm cyc}^{\ell,i}_{\odot}}}$ the equality $h(\varphi_{\odot}^{k(\ell,i)})=k(\ell,i)h (\varphi_{\odot})$ holds. Since $\varphi_{\odot}^{k(\ell,i)}$ fixes each $Y^{\ell,i}_j$ and for each $j$ the mapping
\begin{align}\nonumber
\varphi_{\odot}^{k(\ell,i)}\mid Y^{\ell,i}_j=(\varphi_{\odot}^{j-1}\mid Y^{\ell,i}_1)\circ(\varphi_{\odot}^{k(\ell,i)}\mid Y^{\ell,i}_1) \circ (\varphi_{\odot}^{j-1}\mid Y^{\ell,i}_1)^{-1}
\end{align}
is a conjugate of $\varphi_{\odot}^{k(\ell,i)}\mid Y^{\ell,i}_1$,
the equality
\begin{equation}\label{eq7a.61}
h(\varphi_{\odot}^{k(\ell,i)}\mid Y^{\ell,i}_1)=k(\ell,i) h( \varphi_{\odot})
\end{equation}
holds. The entropy $ h \left( \reallywidehat{{\mathfrak m}_{b,\odot} \mid
{{\rm cyc}^{\ell,i}_{\odot}}} \right)$ is equal to the infimum
\begin{align}\label{eq7a.62}
h \left( \reallywidehat{{\mathfrak m}_{b,\odot} \mid
{{\rm cyc}^{\ell,i}_{\odot}}} \right)& = \inf\{h\big(\varphi_{\odot}\big):\,\varphi_{\odot} \in \reallywidehat{{\mathfrak m}_{b,\odot} \mid
{{\rm cyc}^{\ell,i}_{\odot}}} \}\nonumber\\
& =  \inf\{ \frac{1}{k(\ell,i)}\,\{h\left(\varphi_{\odot}^{k(\ell,i)}\mid Y^{\ell,i}_1\right):\,
\varphi_{\odot}\in  \reallywidehat{{\mathfrak m}_{b,\odot} \mid
{{\rm cyc}^{\ell,i}_{\odot}}}\} \,.
\end{align}
On the other hand,
\begin{equation}\label{eq7a.63}
h(\reallywidehat{{\mathfrak
m}_{b^{k(\ell,i)},\odot} \mid {Y_1^{\ell,i}}})= \inf\{ h(\psi):\, \psi \in\reallywidehat{{\mathfrak
m}_{b^{k(\ell,i)},\odot} \mid {Y_1^{\ell,i}}}\}\,.
\end{equation}
By Lemma \ref{lemm6.5} the inclusion $ \psi \in\reallywidehat{{\mathfrak
m}_{b^{k(\ell,i)},\odot} \mid {Y_1^{\ell,i}}}$ holds if and only if
$\psi=\varphi_{\odot}^{k(\ell,i)} \mid Y_1^{\ell,i}$ for a mapping
$\varphi_{\odot} \in \reallywidehat{{\mathfrak m}_{b,\odot} \mid
{{\rm cyc}^{\ell,i}_{\odot}}} $. 
The lemma is proved. \hfill $\Box$

\bigskip

By Lemma \ref{lemm7.3a} the remaining point of the following theorem is to relate the entropy
of a mapping class on a surface to the entropy of its nodal mapping class.

\begin{thm}\label{thm7.1} For an arbitrary braid the following equality holds
\begin{equation}\label{eq7a.70}
h\left(\widehat{{\mathfrak m}_b} \right) = h(\reallywidehat{\mathfrak{m}_{b,\odot}})=\max_{{\rm
cyc}^{\ell,i}_{\odot}} \left( h \left(\reallywidehat{{\mathfrak
m}_{b^{k(\ell,i)},\odot} \mid {Y_1^{\ell,i}}}\right) \cdot
\frac1{k(\ell,i)} \right) \, .
\end{equation}
Moreover, there exists a mapping with this entropy that represents $\widehat{{\mathfrak m}_b}$.
\end{thm}

We split the proof of Theorem \ref{thm7.1} into two parts. The
first part is the lower bound for the entropy of the braid.
\begin{lemm}\label{lemm7.3}
$$
h \left( \widehat{{\mathfrak m}_{b}} \right)\geq
\max_{{\rm cyc}^{\ell,i}_{\odot}} \left( h \left(
\reallywidehat{{\mathfrak m}_{b^{k(\ell,i)},\odot} \mid
{Y_1^{\ell,i}}}\right) \cdot \frac1{k(\ell,i)} \right) \, .
$$
\end{lemm}
\bigskip

{\bf Proof.} Let $N$ be the smallest natural
number for which $b^N$ is a pure braid. Notice that
\begin{equation}
\label{eq7.43} h({\mathfrak m}_{b^N}) \leq N \cdot h ({\mathfrak
m}_b) \, .
\end{equation}
Indeed, for a chosen homeomorphism $\varphi_b \in {\mathfrak m}_b$
\begin{equation}
\label{eq7.44} h({\mathfrak m}_{b^N}) = {\rm inf} \left\{ h(\varphi)
: \varphi \ \mbox{is  ${\rm Hom}^+(\overline{\mathbb D} ;
\partial \, {\mathbb D} , E)$-isotopic to $\varphi_b^N$}
\right\} \, ,
\end{equation}
and
\begin{equation}
\label{eq7.45} N\cdot h({\mathfrak m}_b)\, = \, {\rm inf} \left\{h( \psi^N)= N
\cdot h(\psi) : \psi \ \mbox{is  ${\rm Hom}^+  (\overline{\mathbb D} ; \partial \, {\mathbb D}
, E)$-isotopic to $\varphi_b$}   \right\} \, .
\end{equation}
Since the infimum in \eqref{eq7.45} is taken over a smaller class
than in \eqref{eq7.44}, we obtain \eqref{eq7.43}.

Suppose ${\rm cyc}^{\ell,i}_{b,\odot}=\big({Y_1^{\ell,i}},\ldots, {Y_{k(\ell,i)}^{\ell,i}}\big)$ is a nodal cycle on the
nodal surface $Y$ such that the mapping class
$\reallywidehat{{\mathfrak m}_{b^{k(\ell,i)},\odot} \mid
{Y_1^{\ell,i}}}$ is pseudo-Anosov, i.e. there exists a
self-homeomorphism of $Y_1^{\ell,i}$ that represents $b^{k(\ell,i)}$ and is conjugate to a non-periodic absolutely extremal mapping $\tilde\varphi^{\ell,i}$. 
Any power of $\tilde\varphi^{\ell,i}$ is again pseudo-Anosov. Hence
the class $\reallywidehat{{\mathfrak m}_{b^N,\odot} \mid
{Y_1^{\ell,i}}}$ is pseudo-Anosov, since $N$ is divisible
by $k(\ell,i)$, i.e.
\begin{equation}\label{eq7.46a}
N = k(\ell,i) \cdot m (\ell,i)
\end{equation}
for an natural number $m (\ell,i)$
Moreover,
\begin{equation}
\label{eq7.46} h \left(\reallywidehat{{\mathfrak
m}_{b^N,\odot} \mid {Y_1^{\ell,i}}} \right) \, = \, h
\left( (\tilde\varphi^{\ell,i})^{m(\ell,i)} \right) = m(\ell,i) \,
h(\tilde \varphi^{\ell,i})
\end{equation}
since $(\tilde \varphi^{\ell,i})^{m(\ell,i)}$ represents
${\mathfrak m}_{b^N,\odot} \mid
{Y_1^{\ell,i}}$. We obtained
\begin{equation}
\label{eq7.47} h \left(\reallywidehat{{\mathfrak
m}_{b^N,\odot} \mid {Y_1^{\ell,i}}} \right) \, = \,
m(\ell,i) \, h \left( \reallywidehat{{\mathfrak
m}_{b^{k(\ell,i)},\odot} \mid {Y_1^{\ell,i}}} \right) \, .
\end{equation}
By Proposition \ref{prop7.2} for the pure braid $b^N$ the inequality $h(\reallywidehat{\mathfrak{m}_{b^N}})\geq
h(\reallywidehat{\mathfrak{m}_{b^N,\odot}\mid {Y_1^{\ell,i}}})$
holds.

By \eqref{eq7.43} and  \eqref{eq7.47} we obtain the inequality
\begin{equation}
\label{eq7.48} N\cdot h(\mathfrak{m}_b)\, \geq \, h\left(
\reallywidehat{{\mathfrak m}_{b^N}} \right) \, \geq \,  h \left(
\reallywidehat{{\mathfrak m}_{b^N,\odot} \mid
{Y_1^{\ell,i}}} \right) \, = \,  m(\ell,i) \, h \left(
\reallywidehat{{\mathfrak m}_{b^{k(\ell,i)},\odot} \mid
{Y_1^{\ell,i}}} \right) \,
\end{equation}
for all $\ell$ and $i$ for which ${\mathfrak m}_{b^{k(\ell,i)},\odot} \mid{Y_1^{\ell,i}}$ is pseudo-Anosov. 
Hence, by \eqref{eq7.46a}
\begin{equation}
\label{eq7.48a}
h(\mathfrak{m}_b)\, \geq \, \frac{1}{k(\ell,i)} h \left(
\reallywidehat{{\mathfrak m}_{b^{k(\ell,i)},\odot} \mid
{Y_1^{\ell,i}}} \right)
\end{equation}
if  ${\mathfrak m}_{b^{k(\ell,i)},\odot} \mid{Y_1^{\ell,i}}$ is pseudo-Anosov.

If $\mathfrak{m}_{b^N,\odot}\mid {Y_1^{\ell,i}}$ is not pseudo-Anosov, it is elliptic, since it is irreducible. In this case the entropy equals zero and inequality \eqref{eq7.48a} is trivially satisfied.
Lemma \ref{lemm7.3}  is
proved. \hfill $\Box$

\bigskip

The remaining inequality that is needed for the proof of Theorem \ref{thm7.1} is stated in the following lemma.

\begin{lemm}\label{lemm7.4} The following inequality holds
$$
h \left( \widehat{{\mathfrak m}_{b}} \right)\leq
\max_{(\ell,i)} \left( h \left(
\reallywidehat{{\mathfrak m}_{b^{k(\ell,i)},\odot} \mid
{Y_1^{\ell,i}}}\right) \cdot \frac1{k(\ell,i)} \right) \, .
$$
\end{lemm}

\medskip

\noindent {\bf Proof of Lemma \ref{lemm7.4}.} The goal is to represent $\mathfrak{m}_b$ by an element $\varphi\in {\rm Hom}^+(\mathbb{P}^1; {\infty},E_n)$ that maps the curve system $\mathcal{C} \subset \mathbb{D}$ onto itself and for each $\ell$ and $i$ the restriction $\varphi^{k(\ell,i)}|S_1^{\ell,i}$  has entropy equal to $h(\reallywidehat{{\mathfrak m}_{b^{k(\ell,i)},\odot} \mid {Y^{\ell,i}}})$, equivalently, the restriction $\varphi| {\rm cyc}_0^{\ell,i}$ has entropy $\frac{1}{k(\ell,i)}h(\reallywidehat{{\mathfrak m}_{b^{k(\ell,i)},\odot} \mid {Y^{\ell,i}}})$. We will use Theorem \ref{thm3.6}.

Suppose $\varphi_{b,\infty}\in {\rm Hom}^+(\mathbb{P}^1;\mathbb{P}^1\setminus \mathbb{D}, E_n)$ represents $\mathfrak{m}_{b,\infty}$ and maps the admissible set of curves $\mathcal{C}$ onto itself.
After a ${\rm Hom}^+(\mathbb{P}^1;\mathbb{P}^1\setminus \mathbb{D}, E_n)$-isotopy of $\varphi_{b,\infty}$ that changes the values only in a small neighbourhood 
 of $\mathcal{C}$
(see Lemma \ref{lem7a.1}) we may assume that for each $\ell >1$ and $i=1,\ldots k_{\ell}$ the iterate $\varphi_{b,\infty}^{k(\ell,i)}$ equals the identity on a small neighbourhood of the exterior boundary  $\partial_{\mathfrak{E}}S_1^{\ell,i}$ of the first labeled set of each cycle ${\rm cyc}_0^{\ell,i}$.
Here $k(\ell,i)$ is the length of the cycle. For uniformity of notation we write $S_1^{1,1}$ for  $S^{1,1}$.

\noindent {\bf A representative $\,{\psi}^{\ell,i}_{\circledcirc}\,$  of the isotopy class of $\,\varphi_{b,\infty}^{k(\ell,i)}\mid S_1^{\ell,i}\setminus E_n'$\,.}
By Corollary \ref{corr2.1} for each $(\ell,i)$ the class $\reallywidehat{\mathfrak{m}_{b^{k(\ell,i)},\odot}|Y^{\ell,i}}$ can be represented by an absolutely extremal self-homeomorphism $\tilde{\psi}_{\odot}^{\ell,i}$ 
of a punctured Riemann surface ${\tilde{Y}^{\ell,i}}$.  The mapping $\tilde{\psi}_{\odot}^{\ell,i}$ has the following properties.
$\tilde{\psi}_{\odot}^{\ell,i}$ minimizes the entropy in the class $\reallywidehat{\mathfrak{m}_{b^{k(\ell,i)},\odot}|Y^{\ell,i}}$.
Moreover, the continuous extension of  $\tilde{\psi}_{\odot}^{\ell,i}$ to the closed Riemann surface $({\tilde{Y}^{\ell,i}})^c$ (also denoted by  $\tilde{\psi}^{\ell,i}_{\odot}$) fixes the node corresponding to the exterior boundary component of $S_1^{\ell,i}$ (in the case of $S^{1,1}_1$ it fixes the node corresponding to $\infty$), and it 
moves the other nodes in $\mathcal{N}\cap ({\tilde{Y}_1^{\ell,i}})^c$
of $({\tilde{Y}_1^{\ell,i}})^c$ along cycles corresponding to the $\varphi_{b^{k(\ell,i)},\,\infty}$-cycles of the connected components of the boundary $\partial S_1^{\ell,i}$.

If the continuous extension of the homeomorphism $\tilde{\psi}^{\ell,i}_{\odot}$ is pseudo-Anosov, then
by Theorem \ref{thm3.6} there exists a self-homeomorphism $\tilde{\psi}^{\ell,i}_{\circledcirc}$ of $({{\tilde{Y}_1^{\ell,i}}})^c$ that is isotopic to $\tilde{\psi}^{\ell,i}_{\odot}$ by an isotopy that changes the mapping only in a punctured neighbourhood of the nodes, that has the same entropy as $\tilde{\psi}^{\ell,i}_{\odot}$, and has the following properties.

$\tilde{\psi}^{\ell,i}_{\circledcirc}$ fixes pointwise a topological disc around the node corresponding to the exterior boundary of
$({{\tilde{Y}_1^{\ell,i}}})^c$.
For each $\tilde{\psi}^{\ell,i}_{\odot}$-cycle 
of nodes there is a $\tilde{\psi}^{\ell,i}_{\circledcirc}$-cycle of closed topological discs around the nodes of 
the cycle, and, for the length of the cycle being equal to $k\geq 1$, the iterate $(\tilde{\psi}^{\ell,i}_{\circledcirc})^{k}$ fixes each disc pointwise.

If the continuous extension of $\tilde{\psi}^{\ell,i}_{\odot}$ is a periodic self-mapping of the Riemann sphere $({\tilde{Y}_1^{\ell,i}})^c$ with distinguished points, then for each cycle of distinguished points of length $k$ the extension of $\tilde{\psi}^{\ell,i}_{\odot}$ itself maps a collection of topological discs around the distinguished points along a cycle of length $k$, and the restriction of the extension of $(\tilde{\psi}^{\ell,i}_{\odot})^k$ to each of the discs is the identity.

Let ${\mathring{Y}_1^{\ell,i}}$ denote the complement in $\tilde{Y}_1^{\ell,i}$ of the collection of all respective open discs. Notice that ${\mathring{Y}_1^{\ell,i}}$ is a bordered surface. There is a homeomorphism $w^{\ell,i}_1: { \mathring{Y}^{\ell,i}_1} \to  \overline{ S^{\ell,i}_1}\setminus E'_n $, for which
the conjugate  ${\psi}^{\ell,i}_{\circledcirc}\stackrel{def}=(w^{\ell,i}_1)\circ (\tilde{\psi}^{\ell,i}_{\odot}\mid {\mathring{Y}_1^{\ell,i}})  \circ (w^{\ell,i}_1)^{-1}$ is a self-homeomorphism of $ \overline{ S^{\ell,i}_1}\setminus {E'}^{\ell,i}_1$, whose restriction to  $S^{\ell,i}_1\setminus {E'}^{\ell,i}_1$ is isotopic to $\varphi^{k(\ell,i)}_{b,\infty}|S^{\ell,i}_1\setminus {E'}^{\ell,i}_1$. Moreover,  ${\psi}^{\ell,i}_{\circledcirc}\mid S^{\ell,i}_1\setminus {E'}^{\ell,i}_1 $ 
also minimizes the entropy among mappings representing $\reallywidehat{{\mathfrak m}_{b^{k(\ell,i)},\odot} \mid
{Y_1^{\ell,i}}}$. In other words,
\begin{align}\label{eq7a.100}
h({\psi}^{\ell,i}_{\circledcirc})=h(\reallywidehat{{\mathfrak m}_{b^{k(\ell,i)},\odot} \mid
{Y_1^{\ell,i}}})\,.
\end{align}
The homeomorphism ${\psi}^{\ell,i}_{\circledcirc}$  is equal to the identity at all points of $\overline{S^{\ell,i}_1}$ that are close to the exterior boundary $\partial_{\mathfrak{E}}S^{\ell,i}_1$ (in a neighbourhood of $\mathbb{P}^1\setminus \mathbb{D}$ 
in case of $S^{1,1}_1$). Moreover, it moves the interior boundary components along  ${\psi}^{\ell,i}_{\circledcirc}$-cycles. If the length of a  ${\psi}^{\ell,i}_{\circledcirc}$-cycle equals $k$ then $({\psi}^{\ell,i}_{\circledcirc})^k$ is the identity near each boundary component of the cycle.
\index{$\psi^{\ell,i}_{\circledcirc}$}

\noindent {\bf The construction of the entropy minimizing element of $\widehat{\mathfrak{m}_b}$.}
\index{$Q_{l}$}
Recall that for $l=1,\ldots,N$ we denoted by $\overline{Q_{l}}$ the following closed subset of $\mathbb{P}^1$ (see \eqref{eq6.0})
$$
\overline{Q_{l}}=\bigcup_{l'=1,\ldots,l}\;\bigcup_{i=1,\ldots,k_{l'}}\;\bigcup_{j=1,\ldots,k(l',i)} \overline{S_j^{l',i}}.
$$
 $\overline{Q_{l}}$ is the closure of a domain $Q_l$ in $\mathbb{P}^1$. Notice that $Q_1= S^{1,1}_1$.

By induction on $l$ we will find a self-homeomorphism $\varphi_{Q_l}$  of $\overline{Q_{l}}$ with distinguished points $E_n'\cap Q_l$  and a self-homeomorphism $\varphi_{b,\infty,l}$ of $\mathbb{P}^1$ for which the following requirements hold.
The entropy of the self-homeomorphism $\varphi_{Q_l}$ equals
$\max_{(l',i):l'\leq l, i\leq k_{l'}} \big( h ({\reallywidehat{\mathfrak{m}_{b^{k(l',i)},\odot} \mid
{Y_1^{l',i}}}}) \cdot \frac1{k(l',i)} \big)$.
The self-homeomorphism $\varphi_{b,\infty,l}$ of $\mathbb{P}^1$
is ${\rm Hom}^+(\mathbb{P}^1;{\infty},{E}_n)$-isotopic to $\varphi_{b,\infty}$ by an isotopy that changes  $\varphi_{b,\infty}$ only in a small neighbourhood of the interior boundary of the $Q_{l'}$ with $l'\leq l$ and coincides with $\varphi_{Q_l}$
on $\partial Q_l$,such that the following holds. 
The homeomorphism $\varphi_{Q_{l}}$
is isotopic to $\varphi_{b,\infty,l}\mid \overline{Q_l}$ by an isotopy
that does not change the mapping in a small neighbourhood of $\partial Q_l \cup (E'_n\cap Q_l)$ and on $\mathbb{P}^1\setminus \mathbb{D}$ (in other words, 
the mapping $\varphi_{Q_{l}}\circ (\varphi_{b,\infty,l}\mid \overline{Q_l})^{-1}$ is ${\rm Hom}\big(\overline{Q_l}; (E_n'\cap Q_l)\cup \partial Q_l\cup (\mathbb{P}^1\setminus \mathbb{D})\big)$-isotopic 
to the identity on $ \overline{Q_l}$. Moreover, 
$\varphi_{Q_{l}}$ equals the identity on $\mathbb{P}^1\setminus \mathbb{D}$ 
and moves the boundary components of $Q_l$ along cycles so that for each $\varphi_{Q_{l}}$-cycle of length $k$ the homeomorphism $\varphi_{Q_{l}}^k$ is the identity near each component of the boundary belonging to the cycle. 

Let first $l=1$, i.e. $Q_1={S_1^{1,1}}$. Both mappings, $\varphi_{b,\infty}|{S_1^{1,1}}$ and ${\psi}^{1,1}_{\circledcirc} |{S_1^{1,1}}$ 
represent
$\reallywidehat{{\mathfrak m}_{b,\odot} \mid {Y_1^{1,1}}}$. Both mappings are equal to the identity on $\mathbb{P}^1\setminus \mathbb{D}$. 
Take any ${\psi}^{1,1}_{\circledcirc}$-cycle of interior boundary components of
${S_1^{1,1}}$. Let $k$ be the length of the cycle. Then
the mapping $({\psi}^{1,1}_{\circledcirc})^k$ equals the identity on each boundary
component of the cycle (and on a small neighbourhood of it in $\overline{S_1^{1,1}}$). Moreover, $h({\psi}^{1,1}_{\circledcirc})=h(\reallywidehat{\mathfrak{m}_{b^{k(1,1)},\odot}\mid Y^{1,1}_1)}$.

Consider a $ {\rm Hom}^+(\mathbb{P}^1;(\mathbb{P}^1\setminus \mathbb{D}), E_n)$-isotopy of $\varphi_{b,\infty}$, which changes the mapping only in a small annulus around the interior boundary components of $S_1^{1,1}$,so that the new self-homeomorphism $\varphi_{b,\infty,1}$ of $\mathbb{P}^1$ with distinguished points $ E_n'$ coincides with 
${\psi}^{1,1}_{\circledcirc}$ 
on $\mathbb{P}^1\setminus \mathbb{D}$ and on all boundary components of $S_1^{1,1}$. 
Then the mapping ${\psi}^{1,1}_{\circledcirc}\circ\varphi_{b,\infty,1}^{-1} $ is the
identity on $\mathbb{P}^1\setminus \mathbb{D}$ and on each boundary component of $S_1^{1,1}$. Since
both mappings, $\varphi_{b,\infty}|S_1^{1,1}$ and ${\psi}^{1,1}_{\circledcirc}|S_1^{1,1}$, represent $\reallywidehat{\mathfrak{m}_{b,\odot}\mid Y^{1,1}}$,
the isotopy classes of their continuous extensions to  $\overline{S_1^{1,1}}$ with distinguished points 
differ by a product of commuting Dehn twists around curves that are homologous to the boundary circles of $S_1^{1,1}$ or to  $\partial \mathbb{D}$. Lemma \ref{lemm3.11}
allows to change the isotopy class of ${\psi}^{1,1}_{\circledcirc}$ by any power of a Dehn twist around the circle  $\partial \mathbb{D}$ 
by changing  ${\psi}^{1,1}_{\circledcirc}$ in a small neighbourhood of this circle without increasing entropy.
The isotopy class of 
$\varphi_{b,\infty,1}\mid \overline{S_1^{1,1}}$ can be changed by any product of Dehn twists around curves that are homologous to the interior boundary circles of $S_1^{1,1}$ by an isotopy of $\varphi_{b,\infty,1}$ on $\mathbb{P}^1$ 
that does not change the mapping outside small neighbourhoods of the interior boundary of $S_1^{1,1}$.
After changing  $\varphi_{b,\infty,1}$ and ${\psi}^{1,1}_{\circledcirc}$ in this way and keeping previous notation, we may assume that 
${\psi}^{1,1}_{\circledcirc}\circ (\varphi_{b,\infty,1}|\overline{S_1^{1,1}})$
is ${\rm Hom}^+(\overline{S_1^{1,1}}; (S_1^{1,1}\cap{E'_n}^{1,1}) \cup
(\mathbb{P}^1\setminus \mathbb{D})
\cup \partial{\overline{S_1^{1,1}}})$-isotopic
to the identity. 
Thus for $\overline{Q_1}\stackrel{def}= \overline{S^{1,1}}$ the homeomorphisms $\varphi_{Q_1}\stackrel{def}={\psi}^{1,1}_{\circledcirc}$ and the homeomorphism $\varphi_{b,\infty,1}$ that was constructed above satisfy the requirements. 

Suppose for some $l\leq N-1$ we found  a self-homeomorphisms $\varphi_{Q_l}$with distinguished points $E_n'\cap Q_l$  and  a self-homeomorphism $\varphi_{b,\infty,l}$
of $\mathbb{P}^1$ with  distinguished points $E_n'$ that satisfy the requirements.
We have to find these objects for the number $l+1$.
\index{$\varphi_{Q_{l}}$}     \index{$\varphi_{b,\infty,l}$}
\index{$\chi^{Q_l}_t$} \index{$k_l$}

The boundary components of $Q_l$ are exactly the curves in $\mathcal{C}$ that are of generation $l+1$. Each such curve 
is the exterior boundary of a set $S^{l+1,i}_j$.
Take any cycle  $\varphi_{b,\infty,l}$-cycle ${\rm cyc}_0^{l+1,i}=({S_1^{l+1,i}},\ldots,S_{k(l+1,i)}^{l+1,i})$ of length $k(l+1,i)$, $1\leq i\leq k_{l+1}$. 
We consider the set $\overline{Q_{l,i}}\stackrel{def}=\overline{Q_l}\cup  \overline{S^{l+1,i}_1}\cup\ldots\cup \overline{S^{l+1,i}_{k(l+1,i)}}$ instead of $\overline{Q_{l+1}}$, and will associate to  $\overline{Q_{l,i}}$ a self-homeomorphism $\varphi_{Q_{l,i}}$ of $\overline{Q_{l,i}}$ and a homeomorphism $ \varphi_{b,\infty,l,i}\in {\rm Hom}^+(\mathbb{P}^1; \mathbb{P}^1\setminus\mathbb{D}, \mathcal{C}\cup E_n')$.

On $\overline{Q_l}$ we put $\varphi_{Q_{l,i}}=\varphi_{Q_{l}} $ and $\varphi_{b,\infty,l,i}= \varphi_{b,\infty,l}$ .
On each $\overline{S^{l+1,i}_j}$ with $j=2,\ldots, k(l+1,i)$ we put
$\varphi_{Q_{l,i}}  = \varphi_{b,\infty,l}$ and $\varphi_{b,\infty,l,i}=\varphi_{b,\infty,l}$.

To define the two mappings on $\overline{S_1^{l+1,i}}$ we consider the self-homeomorphism
${\psi}^{l+1,i}_{\circledcirc}$ of $\overline{S_1^{l+1,i}}$ with distinguished points $E_n'\cap S_1^{l+1,i}$ constructed in the first part of the proof. 
Recall that the mapping ${\psi}^{l+1,i}_{\circledcirc}$ is an entropy minimizing representative of $\reallywidehat{\mathfrak{m}_{b^{k(l+1,i)},\odot}\mid Y^{l+1,i}_1}$, i.e.
$h({\psi}^{l+1,i}_{\circledcirc})=h(\reallywidehat{{\mathfrak m}_{b^{k(l+1,i)},\odot} \mid
{Y_1^{l+1,i}}})$. 
Moreover, the homeomorphism ${\psi}^{l+1,i}_{\circledcirc}$  is equal to the identity at all points of $\overline{S^{l+1,i}_1}$ that are close to the exterior boundary $\partial_{\mathfrak{E}}S^{l+1,i}_1$, and moves the interior boundary components along  ${\psi}^{l+1,i}_{\circledcirc}$-cycles. If the length of a  ${\psi}^{l+1,i}_{\circledcirc}$-cycle equals $k$ then $({\psi}^{l+1,i}_{\circledcirc})^k$ is the identity near each boundary component of the cycle.

Define
\begin{align}\label{eq7a.32}
\varphi_{Q_{l,i}}\stackrel{def}= (\varphi_{b,\infty,l})^{-k(l+1,i)+1}\circ {\psi}^{l+1,i}_{\circledcirc}
\end{align}
on $\overline{S_1^{l+1,i}}$. On the exterior boundary $\partial_{\mathfrak{E}} S_1^{l+1,i}$ the mapping ${\psi}^{l+1,i}_{\circledcirc}$ is equal to the identity, and also   $(\varphi_{b,\infty,l})^{k(l+1,i)}$ is the identity there.
Hence, $\varphi_{Q_{l,i}}= \varphi_{b,\infty,l} $ on   $\partial_{\mathfrak{E}} S_1^{l+1,i}$. Hence, by the induction hypothesis $\varphi_{Q_{l,i}}= \varphi_{Q_l}$ on   $\partial_{\mathfrak{E}} S_1^{l+1,i}$, and  $\varphi_{Q_{l,i}}$ is a well-defined self-homeomorphism of  $Q_{l,i}$.

Since the self-homeomorphism   ${\psi}^{l+1,i}_{\circledcirc}\mid S_1^{l+1,i}\setminus E_n'$ is isotopic to  $ \varphi_{b,\infty,l}^{k(l+1,i)} \mid S_1^{l+1,i}\setminus E_n'   $,
the two mappings that are defined on $\cup_{j=1}^{k(l+1,i)} S_j^{l+1,i}\setminus E_n'$ by $\varphi_ {Q_{l,i}}$, and $ \varphi_{b,\infty,l}$ respectively, represent $\mathfrak{m}_b \mid \cup_{j=1}^{k(l+1,i)} S_j^{l+1,i}\setminus E_n'$. This follows from the Lemma on Conjugation (see Appendix \ref{ChapterB}).

Consider a  ${\rm Hom}^+(\mathbb{P}^1;{\infty},{E}_n)$-isotopy of the self-homeomorphism $\varphi_{b,\infty,l}$ of $\mathbb{P}^1$, which changes the mappings only in a small neighbourhood of the interior boundary components of $S_1^{l+1,i}$, so that for the new self-homeomorphism $\varphi_{b,\infty,l,i}$ of $\mathbb{P}^1$ with set of distinguished points $E_n'$ the restriction of the iterate $\varphi_{b,\infty,l,i}^{k(l+1,i)}\mid \overline{S_1^{l+1,i}}$
coincides with ${\psi}^{l+1,i}_{\circledcirc}$ 
on a small neighbourhood in $\overline{S_1^{l+1,i}}$ of all boundary components of $S_1^{l+1,i}$. 
Equivalently, the restriction $\varphi_{b,\infty,l,i}\mid
\overline{S_1^{l+1,i}}$
coincides with 
$\varphi_{b,\infty,l,i}^{-k(l+1,i)+1}\circ {\psi}^{l+1,i}_{\circledcirc}\mid
\partial_{\mathfrak{E}}S^{l+1,i}_1$ near  all interior boundary
components of $S_1^{l+1,i}$. 
If necessary, we may change ${\psi}^{l+1,i}_{\circledcirc}$  without increasing entropy by a power of a Dehn twist about a circle that is homologous to the exterior boundary $\partial_{\mathfrak{E}}S^{l+1,i}_1$ 
of $\overline{S^{l+1,i}_1}$, and build the  mapping  $\varphi_{Q_{l,i}} $ from the new ${\psi}^{l+1,i}_{\circledcirc}$ (see equation \eqref{eq7a.32}). We use the previous names for the changed objects. 
In the result of the changes both mappings, the self-homeomorphism $\varphi_{Q_{l,i}}$
of $\overline{\Omega^{l+1,i}}\stackrel{def}=\big(\overline{S^{l+1,i}_1}\cup\ldots\cup \overline{S^{l+1,i}_{k(l+1,i)}}\big)$ 
and  the mapping $\varphi_{b,\infty,l,i}\mid \overline{ \Omega^{l+1,i}}$ 
are isotopic on  $\overline{\Omega^{l+1,i}}$ 
with set of distinguished points $ E_n'\cap \overline{\Omega^{l+1,i}}$
by an isotopy that does not change the mappings on the boundary of the set. 
For the obtained a self-homeomorphism $\varphi_{Q_{l,i}}$ of  $ \big(\overline{Q_l}\cup  \overline{S^{l+1,i}_1}\cup\ldots\cup \overline{S^{l+1,i}_{k(l+1,i)}}\big)$
the following equality on entropies hold
\begin{align}\label{eq7a.14'}
h\big(\varphi_{Q_{l,i}}\mid (\overline{S^{l+1,i}_1}\cup\ldots\cup \overline{S^{l+1,i}_{k(l+1,i)}})\big)= \frac{1}{k(l+1,i)} h(\reallywidehat{\mathfrak{m}_{b^{k(l+1,i)},\odot}\mid Y^{l+1,i}_1})\,.
\end{align}
Moreover, the self-homeomorphism $\varphi_{b,\infty,l,i}$ of $\mathbb{P}^1  $ with set of distinguished points $ E_n'$, 
is isotopic to $\varphi_{b,\infty}$ by an isotopy that changes the mapping only in a small neighbourhood of the curves of the system $\mathcal{C}$ of generation at most $l$ and in a neighbourhood of the boundary curves of the set $\bigcup_{j=1,\ldots, k(l+1,i)}S_j^{l+1,i}$. The homeomorphisms $\varphi_{Q_{l,i}}$ and $\varphi_{b,\infty,l,i}\mid \overline{Q_{l,i}}$ 
are isotopic by an isotopy that does not change the values on $ E_n'\cap {Q_{l,i}}$ and on the boundary of  $Q_{l,i}$.

We have to prove the statement concerning the behaviour of $\varphi_{Q_{l,i}}$
on the boundary components of
$\overline{Q_{l,i}}=\overline{Q_l}\cup \overline{S^{l +1,i}_1}\cup\ldots\cup \overline{S^{l+1,i}_{k(l+1,i)}}$. Each such boundary component is a boundary component of $Q_l$ or
an interior boundary component of an ${S^{l+1,i}_j}$. The statement is clear for the boundary components of $Q_l$. It remains to prove it for the interior boundary component of an ${S^{l+1,i}_j}$.

Consider first an interior boundary component of ${S^{l+1,i}_1}$. The mapping ${\psi}^{l+1,i}_{\circledcirc}$ moves it along a ${\psi}^{l+1,i}_{\circledcirc}$-cycle of interior  boundary components of ${S^{l+1,i}_1}$. Denote its length by $k$. Since  ${\psi}^{l+1,i}_{\circledcirc}=\varphi_{Q_{l,i}}^{k(l+1,i)} $ on
$\overline{S^{l +1,i}_1}$, the ${\psi}^{l+1,i}_{\circledcirc}$-cycle of length $k$ is part of a $\varphi_{Q_{l,i}}$-cycle of length $k\cdot k(l+1,i)$, consisting of interior boundary components of ${S^{l +1,i}_1}\cup\ldots\cup {S^{l+1,i}_{k(l+1,i)}}$. All $\varphi_{Q_{l,i}}$-cycles of boundary components of $\overline{Q_l}\cup \overline{S^{l +1,i}_1}\cup\ldots\cup \overline{S^{l+1,i}_{k(l+1,i)}}$  of length $k\cdot k(l+1,i)$, whose first member is a boundary component of  ${S^{l +1,i}_1}$, are obtained in this way. Since  $({\psi}^{l+1,i}_{\circledcirc})^k$ is the identity on the first boundary component of the ${\psi}^{l+1,i}_{\circledcirc}$-cycle, $\varphi_{Q_{l,i}}^{k\cdot k(l+1,i)}$ is the identity on this boundary component.

For any pair of boundary components that are members of the  $\varphi_{Q_{l,i}}$-cycle, the restrictions of $\varphi_{Q_{l,i}}^{k\cdot k(l+1,i)}$ to the components are conjugate to each other. Hence, $\varphi_{Q_{l,i}}^{k\cdot k(l+1,i)}$ is the identity on each such component.

Take any $\varphi_{Q_{l,i}}$-cycle of interior boundary components of $ \overline{S^{l +1,i}_1}\cup\ldots\cup \overline{S^{l+1,i}_{k(l+1,i)}}$. Since $\varphi_{Q_{l,i}}$
moves the $S^{l+1,i}_j$ along the cycle of length $k(l+1,i)$, the  $\varphi_{Q_{l,i}}$-cycle of interior boundary components contains an interior boundary component of $S_1^{l+1,i}$. By the preceding argument the restriction of $\varphi_{Q_{l,i}}^{k\cdot k(l+1,i)}$ to this component is the identity. Hence,   $\varphi_{Q_{l,i}}^{k\cdot k(l+1,i)}$ is equal to the identity on each component of the cycle.

Proceed in this way with all $\varphi_{b,\infty,l}$-cycles ${\rm cyc}_0^{l+1,i'}$. We obtain the mappings $\varphi_{Q_{l+1}}$ and $\varphi_{b,\infty,l+1}$ with the required induction properties.
By the induction hypothesis and Theorem 4 of \cite{AKM}
\begin{align}\label{eq7a.65}
h(\varphi_{Q_{l+1}}) \leq \max_{(l',i):l'\leq l+1, i\leq k_{l'}} \big( h ({\reallywidehat{\mathfrak{m}_{b^{k(l',i)},\odot} \mid
{Y_1^{l',i}}}}) \cdot \frac1{k(l',i)} \big)\,.
\end{align}

If  $l=N-1$ with $N$ being equal to the number if the generations, the sets $S_j^{N,i}$ of the next generation have no interior boundaries.
The process stops and we obtain
a mapping $\varphi_{Q_N}$ that represents $\mathfrak{m}_{b}$ and has entropy 
equal to the right hand side of the equality in the statement of the lemma. The lemma is proved. \hfill $\Box$

\medskip

\begin{rem}\label{rem6.2} The construction of the proof of Lemma $\ref{lemm7.4}$ allows
to recover, up to a product of commuting Dehn twists, the conjugacy class of a mapping class from its irreducible nodal components.
\end{rem}
Indeed, when we only know the irreducible nodal components of the mapping, we make the construction of the proof of Lemma \ref{lemm7.4} except the corrections related to the 
Dehn twists about simple closed loops that are homologous to
the boundary components of the $S_1^{\ell,i}$.

\medskip

\noindent {\bf The Proof of
Theorem \ref{thm7.1}} is a consequence of Lemmas \ref{lemm7.3} and \ref{lemm7.4}. \hfill $\Box$

\medskip
The counterpart of Theorem \ref{thm7.1} for the irreducible braid components is the following theorem.

\begin{thm}\label{thm7.2}
\begin{equation}
\label{eq7.1'} {\mathcal M} (\hat b) = \min_{\ell,j} \left(
{\mathcal M} \left( \widehat{B(\ell,j)} \right) \cdot k(\ell,j)
\right) \, .
\end{equation}
\end{thm}
The following lemma is the lower bound for the conformal module of $\hat b$.
\begin{lemm}\label{lemm7.1}
\begin{equation}
\label{eq7.1a} {\mathcal M} (\hat b) \geq \min_{\ell,j} \left(
{\mathcal M} \left( \widehat{B(\ell,j)} \right) \cdot k(\ell,j)
\right) \, .
\end{equation}
\end{lemm}

\medskip

\noindent {\bf Proof of Lemma \ref{lemm7.1}.} Let
\begin{equation}
\label{eq7.2} M_0 \overset{\rm def}{=} \min_{\ell,j} \left(
{\mathcal M} \left( \widehat{B(\ell,j)} \right) \cdot k(\ell,j)
\right)
\end{equation}
be the right hand side of the inequality (\ref{eq7.1a}) in the lemma.
Let $r_0$ be a positive number such that $r_0 < e^{2\pi M_0}$.
Consider the annulus
$$
A_0 = \left\{ z \in {\mathbb C} : \frac1{\sqrt{r_0}} < \vert z \vert
< \sqrt{r_0} \right\} \, ,
$$
whose conformal module equals
$$
m ( A_0) = \frac1{2\pi} \log r_0 < M_0 \, ,
$$
and an annulus $A \supset \overline{A_0}$, whose conformal module is also less than $M_0$.

Consider for each $\ell$ and $i$ the conjugacy class $\widehat{\mathbold{B}(\ell,i)}$ corresponding to the cycle ${\rm cyc}^{\ell,i}$. Recall that $k(\ell,i)$ is the length of the cycle.
The conjugacy class is represented by closed geometric fat braids in the
product $\widetilde{\partial \mathbb{D}}^{k(\ell,i)} \times \mathbb{C}$ of the $k(\ell,i)$-fold covering of
the unit circle with the complex plane, equivalently, by continuous maps from $\widetilde{\partial \mathbb{D}}^{k(\ell,i)}$ to a symmetrized configuration space. Recall, that we may consider the points of a symmetrized configuration space as subsets of the complex plane $\mathbb{C}$.

We will define fat holomorphic maps similarly as we defined fat geometric braids. A holomorphic  map $\mathbold{F}$ from an annulus $\textsf{A}$ into $C_m(\mathbb{C})\diagup \mathcal{S}_m$ is called a holomorphic fat map if some of the connected components
of the relatively closed complex curve $\{(z,\mathbold{F(z)}), \, z \in \textsf{A}\}$ in $\textsf{A}\times \mathbb{C}$ are declared to be fat.

We will represent now  the conjugacy class $\widehat{\mathbold{B}(\ell,i)}$ by a holomorphic fat map $\widetilde{\mathbold{F}^{\ell,i}}^{\,k(\ell,1)}$ of the $k(\ell,i)$-fold covering ${\widetilde{A}°}^{\,k(\ell,1)}$ of
the annulus $A$ into the respective symmetrized configuration space
$C_{n(\ell,i)}(\mathbb{D})\diagup \mathcal{S}_{n(\ell,i)}$. We choose
the representatives so that we obtain complex curves contained in ${\widetilde{A}°}^{\,k(\ell,1)}\times \mathbb{D}$, not merely in ${\widetilde{A}°}^{\,k(\ell,1)}\times \mathbb{C}$.
Since ${\mathcal M} (\widehat{\mathbold{B}(\ell,i)}) \geq \frac{1}{k(\ell,i) } \cdot M_0 >  \frac{1}{k(\ell,i) }m(A)$ and $\frac{1}{k(\ell,i) }m(A)$ is the
conformal module of the $k(\ell,i)$-fold covering of
the annulus $A$, such a mapping $\widetilde{\mathbold{F}^{\ell,i}}^{\,k(\ell,1)}$  exists.

We will put together these mappings.
Start with the holomorphic fat map $\widetilde{\mathbold{F}^{1,1}}^1$ on $A$. Choose $\varepsilon_2>0$ so that
\begin{equation}\label{eq7a.25}
\{(z, \widetilde{\mathbold{F}^{1,1}}^1   (z)), z\in \overline{A_0}\}
\end{equation}
is a deformation retract of its $\varepsilon_2$-neighbourhood in $\overline{A_0} \times \mathbb{C}$. Consider all fat connected components of \eqref{eq7a.25}. They correspond to the cycles ${\rm cyc}^{2,i}$ and are defined by the fat mappings $\mathbold{F}_{2,i}$.
Define the holomorphic fat map
\begin{equation}\label{eq7a.26}
\widetilde{\mathbold{F}^{1,1}}^1
\sqcup_{\mathbold{F}_{2,1} }(\mathbold{F}_{2,1}\boxplus \varepsilon_2\widetilde{ \mathbold{F}^{2,1}}^{k(2,1)})\sqcup \ldots\sqcup_{\mathbold{F}_{2,k_2} }(\mathbold{F}_{2,k_2}\boxplus \varepsilon_2 \widetilde{\mathbold{F}^{2,k_2}}^{(k(2,k_2)})
\end{equation}
in a neighbourhood of  $\overline{A_0}$
in the same way as we defined the maps that determine the respective geometric fat braids. 
By induction on the number $\ell$ of the generation we find a small positive number $\varepsilon_{\ell+1}$ so that the set defined by
\begin{align}\label{7a.29}
\widetilde{\mathbold{F}^{1,1}}^1
\sqcup_{\mathbold{F}_{2,1} }(\mathbold{F}_{2,1}\boxplus \varepsilon_2\widetilde{ \mathbold{F}^{2,1}}^{k(2,1)})\sqcup \ldots\sqcup_{\mathbold{F}_{2,k_2} }(\mathbold{F}_{2,k_2}\boxplus \varepsilon_2 \widetilde{\mathbold{F}^{2,k_2}}^{(k(2,k_2)})
\nonumber\\
\sqcup \ldots\sqcup_ {\mathbold{F}_{\ell,1}}(\mathbold{F}_{\ell,1}\boxplus \varepsilon_{\ell} \widetilde{\mathbold{F}^{\ell,1}}^{k(\ell,1)})
\sqcup \ldots\sqcup_ {\mathbold{F}_{\ell,k_{\ell}}} (\mathbold{F}_{\ell,k_{\ell}}\boxplus \varepsilon_{\ell}  \widetilde{\mathbold{F}^{\ell,k_{\ell}}}^{k(\ell,k_\ell)})\,
\end{align}
is a deformation retract of its $\varepsilon_{\ell+1}$-neighbourhood in $\overline{A_0} \times \mathbb{C}$. This allows to define the set \eqref{7a.29} with $\ell$ replaced by $\ell +1$.
By the same arguments as in the recovery procedure we obtain for $\ell=N$ a holomorphic map from $A_0$ into the symmetrized configuration space of suitable dimension that extends continuously to the closure $\overline{A_0}$ and represents $\hat b$. Recall that for an arbitrary number $r_0< e^{2\pi M_0}$ we have chosen an annulus $A_0$ of conformal module $\frac{1}{2\pi} \log r_0$ and represented $b$ by a holomorphic mapping of $A_0$ into symmetrized configuration space. Hence, $\mathcal{M}(\hat b)>M_0$, and Lemma \ref{lemm7.1} is proved. \hfill $\Box$

\medskip

The proof of the opposite inequality in the case of pure braids was straightforward, since in that case each mapping that represents a conjugacy class $\hat b$ provides also a mapping that represents a given irreducible braid component by forgetting suitable strands. This does not work in the case of non-pure braids.

We will first prove the Main Theorem and then prove the remaining inequality in Theorem \ref{thm7.2}. For the upper bound of the conformal module of the braid in the Main Theorem we will 
use the estimate of the conformal module by the translation length of the modular transformation which is provided by Proposition \ref{prop4.2} in the general case. Further we will use the relation between the translation length of the modular transformation and the quasiconformal dilatation of the absolutely extremal mapping on the nodal Riemann surface, and its relation to the quasiconformal dilatation of the mappings of cycles of parts of the nodal Riemann surface. This will imply the Main Theorem.
The proof of the remaining inequality of Theorem \ref{thm7.2} will follow  from relations obtained in the proof of the Main Theorem.

\medskip

One of the inequalities of the Main Theorem is stated in the following corollary of  Lemma \ref{lemm7.1}.
\begin{cor}\label{cor7.2}
$$
{\mathcal M} (\widehat b) \geq \frac\pi2 \, \frac1{h(\widehat{\mathfrak
m_b})} = \frac\pi2 \, \frac1{h(\widehat b)} \, .
$$
\end{cor}

\noindent {\bf Proof.} 
By the equality \eqref{eq7a.20''}, and the Main Theorem in the irreducible case we obtain for each $\ell$ and $i$
\begin{equation}
\label{eq7.49} k(\ell,i) \, {\mathcal M} \left( \widehat{B(\ell,i)}
\right) = \frac\pi2 \, \frac{k(\ell,i)}{h \left(
\reallywidehat{{\mathfrak m}_{b^{k(\ell,i)},\odot} \mid
{Y_1^{\ell,i}}} \right)} \, .
\end{equation}

Lemma \ref{lemm7.1} and Theorem \ref{thm7.1} imply the inequality
$$
{\mathcal M} (\widehat b) \geq \frac\pi2 \, \frac1{h(\widehat{\mathfrak
m}_b)} = \frac\pi2 \, \frac1{h(\widehat b)} \,
$$
that is stated in the corollary. \hfill $\Box$

\medskip

To prove the Main Theorem in the general case it remains to prove the opposite inequality that is stated in Lemma \ref{lemm7.6} below.

\begin{lemm}\label{lemm7.6}
$$
{\mathcal M} (\widehat b) \leq \frac\pi2 \, \frac1{h(\widehat b)} \, .
$$
\end{lemm}
\medskip

\noindent {\bf Proof.} By Proposition \ref{prop4.2}
$$
{\mathcal M} (\widehat b) \leq \frac\pi2 \,
\frac1{L(\varphi_{b,\infty}^*)} \, .
$$
Here again $\varphi_{b,\infty}$ is a self-homeomorphism of
$\mathbb{P}^1$ with the following properties. It is equal to the
identity outside $\overline{\mathbb{D}}$ and on
$\overline{\mathbb{D}}$ it is equal to a homeomorphism $\varphi_b$
representing the mapping class $\mathfrak{m}_b.\,$ As before
$\varphi_{b,\infty}^*$ denotes the modular transformation of
$\varphi_{b,\infty}.$ We have
$$
L(\varphi_{b,\infty}^*) = \frac12 \log I(\varphi_{b,\infty})
$$
(see (\ref{eq2.32}) and (\ref{eq2.33})). As in Chapter \ref{chapter2},
$I(\varphi_{b,\infty})$ is the infimum of quasiconformal dilations
in the class obtained from $\varphi_{b,\infty}$ by isotopy and
conjugation.

We may assume that $\varphi_{b,\infty}$ is completely reduced by an admissible
system of curves $\mathcal{C}\,$ that are contained in the disc $\mathbb{D}$. Let $Y$ be a nodal surface
associated to the system $\mathcal{C}.\,$ Let $\tilde w: Y \to
\tilde w(Y) = \tilde Y$ be the conformal structure of part (1) of
Theorem \ref{thm2.17} and let $\tilde{\varphi}$ be the absolutely
extremal self-homeomorphism of $\tilde Y$ which appears in (1),
Theorem \ref{thm2.17}. We have
$$
\frac{1}{2} \, \log K(\tilde{\varphi})\, =\, \frac{1}{2} \, \log
I({\varphi}_{b,\infty}) \,. 
$$
Then for the $\tilde{\varphi}$-cycles ${\rm
cyc}^{\ell,i}_{\odot}$ of $\tilde Y$ we have
\begin{equation}
\label{eq7.50} \frac{1}{2} \, \log K(\tilde{\varphi}) \, = \,
\underset{\ell,i} {\max}\left( \frac{1}{2} \, \log
K(\tilde{\varphi}\mid {\rm
cyc}^{\ell,i}_{\odot}
) \right).
\end{equation}
By Lemma \ref{lemm2.16}
\begin{equation}
\label{eq7.51} \frac{1}{2} \, \log K(\tilde{\varphi}\mid
{\rm
cyc}^{\ell,i}_{\odot})\, = \, \frac{1}{2}
\,\frac{1}{k(\ell,i)}\, \log K(\tilde{\varphi}^{k(\ell,i)} \mid
\tilde Y^{\ell,i}_1)\,.
\end{equation}
By the Theorem of Fathi-Shub for irreducible self-homeomorphisms of
connected Riemann surfaces of first kind (see Theorem \ref{thm3.2}) we obtain
\begin{equation}
\label{eq7.52} \frac{1}{2} \,\frac{1}{k(\ell,i)}\, \log
K(\tilde{\varphi}^{k(\ell,i)} \mid \tilde
Y^{\ell,i}_1)\,=\,\frac{1}{k(\ell,i)}\,
h(\tilde{\varphi}^{k(\ell,i)} \mid \tilde Y^{\ell,i}_1).
\end{equation}
Notice that
\begin{equation}\label{eq7.53}
h(\tilde{\varphi}^{k(\ell,i)} \mid \tilde Y^{\ell,i}_1) \, = \, h
\left(\reallywidehat{{\mathfrak{m}}_{b^{k(\ell,i)},\odot}
\mid \tilde{Y}_1^{\ell,i}} \right),
\end{equation}
since $\tilde{\varphi}^{k(\ell,i)}$ is an absolutely extremal element of the
class ${\mathfrak{m}}_{b^{k(\ell,i)},\odot}$ induced
by $\varphi_{b,\infty}$ on $\tilde Y$. By Theorem \ref{thm7.1}
\begin{equation}
\label{eq7.54} \underset{\ell,i}{\max}\; \frac{1}{k(\ell,i)} \;
h\left(\reallywidehat{{\mathfrak{m}}_{b^{k(\ell,i)},\odot}\mid
\tilde{Y}_1^{\ell,i}}\right) \,=\,
h\left(\widehat{{\mathfrak{m}}_{b,\odot}}\right)=h(\widehat{\mathfrak{m}_b})= h(\hat b)\,.
\end{equation}
Hence, by \eqref{eq7.50}, \eqref{eq7.51}, \eqref{eq7.52},
\eqref{eq7.53} and \eqref{eq7.54} we have
$$
\frac{1}{2} \log I(\varphi_{b,\infty}) = \frac{1}{2} \log
K(\tilde{\varphi})=h\left(\widehat{{\mathfrak{m}}_{b,\odot}}\right). 
$$
Hence
\begin{align}\label{eq7.55}
\mathcal{M}(\hat b) \leq \frac{\pi}{2}\,
\frac{1}{h(\widehat{{\mathfrak{m}}_{b,\odot}}) }\, =
\,\frac{\pi}{2}\, \frac{1}{h(\widehat{\mathfrak{m}_b})}\,=\,
\frac{\pi}{2}\, \frac{1}{h(\widehat b)}.
\end{align}
Lemma \ref{lemm7.6}, and, hence, the general case of the Main Theorem is
proved. \hfill $\Box$

\bigskip

The following corollary of Lemma \ref{lemm7.6} states the remaining inequality that is needed for the proof of Theorem \ref{thm7.2}. 
\begin{cor}\label{cor7.1}
$$
{\mathcal M} (\widehat b) \leq  \min_{\ell,i} \left( {\mathcal M} \left(
\widehat{B(\ell,i)} \right) \cdot k (\ell,i) \right) \, .
$$
\end{cor}
\medskip

\noindent {\bf Proof.} The corollary follows from equations \eqref{eq7.49} and \eqref{eq7.54} and inequality \eqref{eq7.55}.
\hfill $\Box$
\bigskip

Theorem \ref{thm7.1} together with the Main Theorem imply the following corollary.
\begin{cor}\label{cor7a.0}
For each $\hat b \in \hat
{\mathcal{B}}_n$ $(n\geq 2)$ and each nonzero integer $l$
$$
\mathcal{M}(\widehat{ b\,^l})=\frac{1}{|l|}\mathcal{M}(\hat b)\,.
$$
\end{cor}

\medskip

\noindent {\bf Proof of Corollary \ref{cor7a.0}.} By the Main Theorem Corollary  \ref{cor7a.0} is equivalent to the equality
$$
h \left(\widehat{{\mathfrak m}_{b^{\ell}}} \right) = \vert \ell
\vert \, h \left(\widehat{{\mathfrak m}_b} \right)
$$
for each braid and each non-zero integer. The equality is easy for
irreducible braids. Indeed, if the class $\widehat{\mathfrak m}_b$ is
elliptic, the class $\widehat{{\mathfrak m}_{b^{\ell}}}$ is so and
both sides are zero. If the class $\widehat{\mathfrak m}_b$ is
pseudo-Anosov and $\tilde\varphi$ is a pseudo-Anosov representative
then $\tilde\varphi^{\ell}$ is a pseudo-Anosov representative of
$\widehat{{\mathfrak m}_{b^{\ell}}}$. The equality follows from
$$
h(\tilde\varphi^{\ell}) = \vert \ell \vert \, h (\tilde\varphi) \, .
$$
(See \cite{AKM}, Theorem 2 or Section \ref{sec:entropy.1}).

If $b$ is reducible one has to apply this argument to each cycle of
parts of a nodal surface, that is associated to the  class ${\mathfrak m}_b$
and an admissible system of curves ${\mathcal C}$ which completely
reduces a representative of ${\mathfrak m}_b$, and to use Theorem \ref{thm7.1}.
\hfill $\Box$

\bigskip

\begin{probl}
Define irreducible nodal components of any mapping class on a Riemann surface of genus $g$ with
$m$ punctures, and prove the analogue of Theorem $\ref{thm7.1}$.
\end{probl}

\chapter{The conformal module and holomorphic families of polynomials.}
\label{chapter8}
\setcounter{equation}{0}

In this Chapter we discuss the first result \cite{GL} which uses the concept of the conformal module of conjugacy classes of braids. The paper  \cite{GL} was motivated
by the interest of the authors in Hilbert's Thirteen's Problem and treats global reducibility of holomorphic families of polynomials. 
In Section \ref{sec:8.1} we  give a historical account, in Section \ref{sec:8.2} we give a conceptional  alternative proof of a slightly improved version of the result of \cite{GL}.

\section{Historical remarks.}
\label{sec:8.1}
The conformal module of braids and first applications of this
concept to families of polynomials, depending holomorphically on a parameter, appeared in a line of research that
was initiated by the 13$^{\rm th}$ Hilbert problem, \index{Thirteen's Hilbert problem} posed by Hilbert in his famous talk on the
International Congress of Mathematicians~1900. Hilbert asked whether
each holomorphic function in three variables can be written as
finite composition of continuous functions of two variables. The
question was answered affirmatively by Kolmogorov and Arnold. In
the Proceedings of the Symposium \cite{PPM} devoted to
``Mathematical developments arising from Hilbert Problems'', the
problem is commented as follows.

\medskip

\begin{itemize}
\item[ ] ``Hilbert posed this question especially in connection with the solution of a general
algebraic equation of degree $7$. It is reasonable to presume that he
formulated it in terms of continuous functions partly because he had
an interest in nomography and partly because he expected a negative
answer. Now that it is settled affirmatively, one can ask an equally
fundamental, and perhaps more interesting, question with algebraic
functions instead of continuous functions.''
\end{itemize}

\medskip

This latter question goes back to mathematics of the $17$-th century.
Consider an algebraic equation of degree $n$.
\begin{equation}
\label{eq8.1a} z^n + a_{n-1} \, z^{n-1} + \ldots + a_0 = 0 \, .
\end{equation}
Put $a = (a_0 , \ldots , a_{n-1}) \in {\mathbb C}^n$ and consider
$$
P (z,a) = z^n + a_{n-1} \, z^{n-1} + \ldots + a_0
$$
as a monic polynomial of degree $n$ whose coefficients are polynomials in
$a$. $P$ is also called an algebraic function of degree $n$ on
${\mathbb C}^n$.
Recall that  $\overline{{\mathfrak
P}}_n$ is the space of monic polynomials of degree $n$ (maybe, with multiple zeros). The adjective monic refers to the property that the coefficient by the highest order of the variable equals $1$.
Parameterizing the space $\overline{{\mathfrak
P}}_n$ by unordered $n$-tuples of zeros of polynomials, we get a map
which assigns to each $a \in {\mathbb C}^n$ the $n$-tuple $\{ z_1
(a) , \ldots , z_n (a) \}$ of solutions of the equation $P (z,a) =
0$, i.e. what is called an ``$n$-valued algebraic function''.

Tschirnhaus \index{Tschirnhaus} suggested to ``substitute'' the
unknown variable $z$ in \eqref{eq8.1a} by an algebraic function of a
new variable $w$. In other words, put
\begin{equation}
\label{eq8.1b} w = z^k + b_{k-1} \, z^{k-1} + \ldots + b_0 \, ,
\end{equation}
and let $b = (b_0 , \ldots , b_{k-1}) \in {\mathbb C}^k$. To
eliminate $z$ from the two equations (\ref{eq8.1a}) and
(\ref{eq8.1b}) consider $P$ and $Q$, $Q (z,w) \overset{\rm def}{=}
z^k + b_{k-1} \, z^{k-1} + \ldots + b_0 - w$, as polynomials in $z$
with coefficients being polynomials in $w$. They are relatively
prime. The resultant $R$ is a polynomial in $w$ (and in the
coefficients $a$ and $b$). It can be written as
$$
R = p P + q Q
$$
where $p$ and $q$ are polynomials in $z$ and $w$ (and in the
coefficients $a \in {\mathbb C}^n$, $b \in {\mathbb C}^k$). If $z$
satisfies (\ref{eq8.1a}) and (\ref{eq8.1b}) then $R(w) = 0$.
Vice versa, if  $Q (z,w) =0$ and $R(z)=0$, then $P(z)=0$.
This
allows to think about solutions of (\ref{eq8.1a}) as part of zeros
of a ``composition of algebraic functions'' determined by the
equations $Q(z,w) = 0$, respectively by $R(w) = 0$. The method
consists now in choosing the coefficients $b$ depending on $a$ so
that the ``algebraic functions'' depend on a smaller number of
variables. This method brings the general equation of degree~7 to
the equation
\begin{equation}
\label{eq8.1c} z^7 + a \, z^3 + b \, z^2 + c \, z + 1 = 0 \, .
\end{equation}
The $7$-tuple of solutions of (\ref{eq8.1c}) is an algebraic
function of degree~7 depending on three complex variables. Hilbert
was interested in the ``complexity'' of the general $7$-valued
algebraic function, in other words, he wanted to know how far one
can go in this process to get formulas beyond radicals.

There are two difficulties. First, the ``function-composition'' may
have more zeros than $P$, and, secondly, the ``algebraic functions''
obtained by choosing the coefficients $b$ are actually determined by polynomials
in one variable with coefficients being rational functions in
several complex variables, that may have indeterminacy sets. The
composition problem was rigorously formulated for entire algebraic
functions (i.e. for polynomials in one variable with coefficients
being entire functions of complex variables -- the restricted
composition problem) and was considered first for the case when the
``function-composition'' coincides with the original function (the
faithful composition problem). Arnold's interest in the topological
invariants of the space ${\mathfrak P}_n$ of monic polynomials of degree $n$ without multiple zeros was motivated by finding
cohomological obstructions for the restricted faithful composition
problem. The restricted composition problem is now answered
negatively in a series of papers by several authors. The last step
was done by Lin (see \cite{Li1}). But the original question allowing
polynomials in one variable with coefficients being rational
functions is widely open. Compare also with the more classical
account \cite{Li2}.

On the other hand,  $\mathfrak{P}_n$ is a
complex manifold, in fact, the complement of the algebraic
hypersurface $\{\textsf{D}_n =0\}$ in complex Euclidean space
$\mathbb{C}^n$ \index{$\mathbb{C}^n$}. Here $\textsf{D}_n$
\index{$\textsf{D}_n$} denotes the discriminant of polynomials in
$\mathfrak{P}_n$. Recall that the function $\textsf{D}_n$ is a polynomial in
the coefficients of elements of $\mathfrak{P}_n$. While, in connection with his interest in the Thirteen's Hilbert Problem, Arnold studied
the topological invariants of the space $\mathfrak{P}_n$, 
the conformal invariants of   $\mathfrak{P}_n$  (i.e. invariants that are preserved under biholomorphic mappings) are of interest in connection with Gromov's Oka Principle.
The collection of conformal modules of conjugacy classes of elements of the fundamental group of the space of monic polynomials of degree $n$ without multiple zeros  is a collection of conformal invariants of $\mathfrak{P}_n\cong \mathbb{C}_n\setminus \textsf{D}_n$.
Notice, that for any complex manifold the conformal module of conjugacy classes
of its fundamental group can be defined. The collection of conformal
modules of all conjugacy classes is a biholomorphic invariant of the
manifold. We will focus here on the respective invariants of the space $\mathfrak{P}_n$, in other words, on the conformal module of conjugacy classes of braids.

In whatever sense one understands ``algebraic functions'', for each
algebraic function there is a Zariski open set in ${\mathbb C}^n$
such that the restriction of the algebraic function to this set is a
separable algebroid function. The restriction of the algebraic
function to each loop in this set defines a conjugacy class of
braids. Moreover, this conjugacy class is represented by a
holomorphic mapping of an annulus into the space of polynomials --
any annulus which is mapped  holomorphically 
into the Zariski open
set and represents the homotopy class of the loop may serve. We do
not know at the moment whether further progress related to the
concept of the conformal module of braids may have some impact on
the open problems related to the 13$^{\rm th}$ Hilbert problem or stimulated
by it. I am
grateful to M.~Zaidenberg who asked me this question.

\smallskip
\section{Families of polynomials, solvability and reducibility.}
\label{sec:8.2}

The concept of the conformal module of conjugacy classes of braids
appeared (without name) in the paper \cite{GL}. The paper was motivated
by the interest of the authors in Hilbert's Thirteen's Problem and treats global  reducibility of holomorphic families of polynomials.

The following objects related to $\mathfrak{P}_n$ have been
considered in this connection.
A continuous mapping from a
topological space $X$ into the set of monic polynomials of fixed
degree $\overline{\mathfrak{P}}_n$ (maybe, with multiple zeros) is a quasipolynomial.
\index{quasipolynomial} It can be written as a function in two
variables $x \in X,\; \zeta \in \mathbb{C}:\;$ $\,f(x,\zeta)= a_0(x)
+ a_1(x)\zeta+ ... + a_{n-1}(x)\zeta^{n-1}
+ \zeta^n, \,$
for continuous functions $a_j, \; j=1,...,n\,,$ on $X$. If $X$ is a
complex manifold and the mapping is holomorphic it is called an
algebroid function. If the image of the map is contained in the space
$\mathfrak{P}_n$ of monic polynomials of degree $n$ without multiple zeros, it is called separable. A separable quasipolynomial
\index{quasipolynomial ! separable} is called solvable
\index{quasipolynomial ! solvable} if it can be globally written as
a product of quasipolynomials of degree $1$.
It is called
irreducible \index{quasipolynomial ! irreducible} if it can not be
written as product of two quasipolynomials of positive degree, and reducible otherwise.
We also call a solvable quasipolynomial a globally solvable family of polynomials, and an irreducible quasipolynomial a globally irreducible family of polynomials.
Two
separable quasipolynomials are isotopic if there is a continuous
family of separable quasipolynomials \index{isotopy ! of
quasipolynomials} joining them. An algebroid function
\index{algebroid function} on $\mathbb C ^n$ whose coefficients are
polynomials is called an algebraic function. 
In this terminology the result of the paper \cite{GL} concerns reducibility of holomorphic quasipolynomials.

The space $\mathfrak{P}_n$ \index{$\mathfrak{P}_n$} of
monic polynomials of degree $n$ without multiple zeros coincides with the symmetrized configuration space $C_n(\mathbb{C})\diagup \mathcal{S}_n$. The space $\mathfrak{P}_n$
can be parameterized either by the coefficients or by the unordered
tuple of zeros of polynomials.

A separable quasipolynomial of degree $n$ on a connected closed Riemann surface or a connected bordered Riemann surface $X$ is a continuous mapping from $X$ into the symmetrized configuration space $C_n(\mathbb(C)\diagup \mathcal{S}_n$. The restriction of the quasipolynomial to a loop in $X$ with a base point $x_0$ is a geometric braid. Hence, the monodromy of the mapping to $C_n(\mathbb(C)\diagup \mathcal{S}_n$ along each element of the fundamental group  $\pi_1 (X,q_0)$ of $X$ is a braid. 
Theorem \ref{thmEl-1} in this situation can be stated as follows (see also \cite{St}, \cite{Ha})

\noindent {\it 
The (free) isotopy classes of separable quasipolynomials of degree $n$ on $X$ are in one-to-one correspondence to conjugacy classes of homomorphisms from $\pi_1 (X,q_0)$ into ${\mathcal B}_n$.}

We start with a simple Lemma on global solvability of holomorphic families of polynomials.

\begin{lemm}\label{lem8.0}  Let $X$ be a closed Riemann surface
of positive genus with a closed smoothly bounded topological disc removed. Suppose $f$ is an
irreducible separable algebroid function of degree $3$ on $X$.
Suppose $X$ contains a domain $A$, one of whose boundary components
coincides with the boundary circle of $X$, such that $A$ is
conformally equivalent to an annulus of conformal module strictly
larger than $\frac{\pi}{2}\, (\log( \frac{3 +\sqrt{5}}{2}))^{-1}$.
Then $f$ is solvable over $A$.
\end{lemm}

We postpone the proof of Lemma \ref{lem8.0}, and state some simple lemmas which will be needed in the sequel. Let $A$ be an annulus equipped with an orientation (called positive) of simple dividing closed curves. (A simple closed curve is dividing if its complement is not connected.)  Let $\gamma$ be a positively oriented simple closed dividing curve in $A$. 
For a continuous mapping $f$ from an annulus $A$ to the symmetrized configuration space $C_n(\mathbb{C})\diagup \mathcal{S}_n$ the restriction  $f\mid \gamma$ represents a conjugacy class of braids which is denoted by $\widehat b_{f,A}$. (It does not depend on the choice of the
positively oriented simple closed dividing  curve.)
\index{$\widehat{b}_{f,A}$} \index{curve ! simple closed dividing}

\begin{lemm}\label{lem8.1} Suppose the separable quasipolynomial $f$ of degree
$n$ on an annulus $A$ is irreducible and $n$ is prime. Then the
induced conjugacy class of braids $\widehat b_{f,A}$ is irreducible.
\end{lemm}


Notice that, on the other hand, conjugacy classes of irreducible
pure braids define solvable, hence reducible, quasipolynomials on
the circle.

\medskip

\noindent {\bf Proof.} Recall that $\tau_n:\mathcal{B}_n\to \mathcal{S}_n$ is the natural projection from the braid group to the symmetric group. $\tau_n$ maps conjugacy classes of braids to conjugacy classes of permutations. If $f$ is irreducible the conjugacy class of
braids $\widehat b_{f,A}$ projects to a conjugacy class $\tau_n(\widehat b_{f,A})$ of
$n$-cycles. The lemma is now a consequence of the following known
lemma (see e.g. \cite{FaMa}). \hfill $\Box$

\begin{lemm}\label{lemm8.2} If $n$ is prime then any braid $b \in {\mathcal
B}_n$, for which $\tau_n (b)$ is an $n$-cycle, is irreducible.
\end{lemm}

For convenience of the reader we give the short argument.

\smallskip

\noindent {\bf Proof of Lemma \ref{lemm8.2}.} If $b$ was reducible
then a homeomorphism $\varphi$ which represents the mapping class
corresponding to $b$ would fix setwise an admissible system of
curves ${\mathcal C}$. Let $C_1$ be one of the curves in ${\mathcal
C}$ and let $\delta_1$ be the topological disc contained in
${\mathbb D}$, bounded by $C_1$. $\delta_1$ contains at least two
distinguished points $z_1$ and $z_2$. Since $\varphi$ permutes the
distinguished points along an $n$-cycle there is a power
${\varphi}^k$ of $\varphi$ which maps $z_1$ to $z_2$. Since $n$ is
prime, $\varphi^k$ also permutes the distinguished points along an
$n$-cycle. Hence it maps some distinguished point in $\delta_1$ to
the complement of $\delta_1$. We obtained that $\varphi^k
(\delta_1)$ intersects both, $\delta_1$ and its complement. Hence
$\varphi^k (C_1) \ne C_1$ but $\varphi^k (C_1)$ intersects $C_1$.
This contradicts the fact that the system of curves ${\mathcal C}$
was admissible and invariant under $\varphi$. \hfill $\Box$

\bigskip

Recal that $\Delta_3=\sigma_1\sigma_2\sigma_1= \sigma_2\sigma_1\sigma_2$ and the group $\langle \Delta^2\rangle$ generated by $\Delta^2$ coincides with the center $\mathcal{Z}_3$ of the braid group $\mathcal{B}_3$.  We need the following lemma.
\begin{lemm}\label{lemEl.0}
A mapping class in $\mathfrak{M}(\overline{\mathbb{D}}; \partial \mathbb{D}, E_3)$ is reducible if and only if the braid corresponding to it
is conjugate to $\sigma_1^k\Delta_3^{2\ell}$ for integers $k$ and $\ell$.

A mapping class
in $\mathfrak{M}(\mathbb{P}^1; \{\infty\}, E_3)$ is reducible if and only if it corresponds to a conjugate of $\sigma_1^k\diagup \mathcal{Z}_3$.
\end{lemm}
\noindent {\bf Proof.} Any admissible system of curves in $\overline{{\mathbb{D}}}$ with three distinguished points consists of one curve that divides $\overline{\mathbb{D}}$ into two connected components.
One of these components contains $\partial \mathbb{D}$ and one of the distinguished points, the other one contains two distinguished points.
For a reducible mapping class $\mathfrak{m}$ in $\mathfrak{M}(\overline{\mathbb{D}}; \partial \mathbb{D}, E_3)$ we take an admissible curve $\gamma$ that reduces  $\mathfrak{m}$ and a representing homeomorphism $\varphi\in \mathfrak{m}$ that fixes $\gamma$ pointwise. Suppose the points $\zeta_1,\zeta_2,$ and $\zeta_3$ in $E_3$ are labelled by using a homomorphism from
$\pi_1(C_3(\mathbb{C})\diagup \mathcal{S}_3, E_3)$ to $\pi_1(C_3(\mathbb{C})\diagup \mathcal{S}_3, E^0_3)$ (which is defined up to conjugation) so that one of the connected components $\mathcal{C}_1$ of the complement of $\gamma$ contains $\zeta_1$ and $\zeta_2$, the other connected component $\mathcal{C}_2$ contains $\zeta_3$ and $\partial \mathbb{D}$.
Since $\varphi$ fixes $\partial \mathbb{D}$ and $\gamma$ pointwise, it also fixes the annulus $\overline{\mathcal{C}}_2$ setwise and fixes the point $\zeta_3$.
A self-homeomorphism of an annulus that fixes the boundary pointwise is a twist.
After an isotopy of $\varphi$ on $\mathbb{D}$ that fixes $\gamma$ setwise and fixes each distinguished point we may assume that $\varphi$ fixes pointwise a simple arc that joins $\zeta_3$ with $\gamma$. Then
the restriction $\varphi\mid \overline{\mathcal{C}}_2$ represents a power of a Dehn twist in $\mathbb{D}$ about a curve in $\mathcal{C}_2$ that is homotopic in ${\overline{\mathcal{C}}}_2\setminus \{\zeta_3\}$ to $\partial \mathbb{D}$.
Such a Dehn twist corresponds to the braid $\Delta_3^2$.
The restriction $\varphi\mid \overline{\mathcal{C}}_1$ represents the class in $\mathfrak{M}(\overline{\mathcal{C}}_1;\partial \mathcal{C}_1, \{\zeta_1,\zeta_2\})$, corresponding to $\sigma^k$ for some integer $k$.
Hence the class represented by $\varphi$ in
$\mathfrak{M}(\overline{\mathbb{D}}; \partial \mathbb{D}, E_3)$ corresponds to a braid that is conjugate to $\sigma_1^k \Delta_3^{2\ell}$ for some integer $\ell$. Vice versa, it is clear that the mapping class $\mathfrak{m}_{\sigma_1^{2k} \Delta_3^{2\ell}}$ is reducible.

The statement concerning the reducible elements of $\mathfrak{M}(\mathbb{P}^1; \{\infty\}, E_3)$ is proved in the same way.
\hfill $\Box$

\medskip

We want to point out that the identity in the braid group $\mathcal{B}_n$ with $n\geq 3$ is reducible, the identity in $\mathcal{B}_2$ is irreducible.
Recall that by Example 5 in Section  \ref{sec:4.1b} the value   $\eta = \frac\pi2 \left(\log \frac{3+\sqrt 5}2
\right)^{-1}$ is the largest finite conformal module
among irreducible $3$-braids, and hence, by the Main Theorem  the value $\log \frac{3+\sqrt 5}2$ is the smallest
non-vanishing entropy among $3$-braids. For the latter fact see also  \cite{Son}.

\begin{lemm}\label{lemm8.3} Suppose for a conjugacy class of braids $\widehat b
\in \widehat{\mathcal B}_3$ the conformal module ${\mathcal M}
(\widehat b)$ satisfies the inequality ${\mathcal M} (\widehat b) >
\eta$. Then the following holds.
\begin{enumerate}
\item[(a)] ${\mathcal M} (\widehat b) = \infty$.
\item[(b)] {\it If $\widehat b$ is irreducible then $\widehat b$ is the conjugacy class of a
periodic braid, i.e. either of $(\sigma_1 \, \sigma_2)^{\ell}$ for
an integer $\ell$ not divisible by $3$, or of $(\sigma_1 \, \sigma_2
\, \sigma_1)^{\ell}$ for an integer $\ell$ not divisible by $2$.}
\item[(c)] {\it If $\widehat b$ is reducible then $\widehat b$ is the conjugacy class of $\sigma_1^k \, \Delta_3^{2\ell}$ for integers $k\neq 0$ and $\ell$. Here $\Delta_3^2 = (\sigma_1 \, \sigma_2 \, \sigma_1)^2 = (\sigma_1 \, \sigma_2)^3$.}
\item[(d)] {\it If $\widehat b$ is the conjugacy class of a commutator, then it is represented by a pure braid $b$.}
\end{enumerate}
\end{lemm}

\noindent {\bf Proof.} Suppose $\widehat b$ is irreducible. Since $\frac\pi2 \left(\log \frac{3+\sqrt 5}2 \right)^{-1}$ is the largest finite conformal module
among irreducible $3$-braids the equality ${\mathcal M} (\widehat b) = \infty$ holds. (Hence, also  $h(\widehat b) = 0$.) This proves (a) in
the irreducible case.

Let $b \in {\mathcal B}_3$ represent the irreducible class $\widehat
b$, let ${\mathfrak m}_b\in \mathfrak{M}(\overline{\mathbb{D}};\partial\mathbb{D},E_3)$ be the mapping class corresponding to $b$, and let ${\mathcal H}_{\infty} ({\mathfrak
m}_b)\in \mathfrak{M}(\mathbb{P}^1;\infty,E_3)$ be the corresponding mapping class on $\mathbb{P}^1$ with distinguished points. Since $\mathcal{M}(\widehat{b})=\infty$, the class ${\mathcal H}_{\infty} ({\mathfrak
m}_b)$ is represented by a periodic mapping of the
complex plane ${\mathbb C}$ with three distinguished points (see \cite{Be1}). By a conjugation we may assume that the periodic mapping fixes zero and
hence is a rotation by a root of unity. If zero is not a
distinguished point, the three distinguished points are
equidistributed on a circle with center zero and the mapping is
rotation by a power of $e^{\frac{2\pi i}3}$. The mapping representing  ${\mathcal H}_{\infty} ({\mathfrak
m}_b)$
corresponds to
$
\reallywidehat{(\sigma_1 \, \sigma_2)^{\ell} / \langle \Delta_3^2 \rangle}
\,
$
and the mapping representing  ${\mathfrak m}_b$ corresponds to $\reallywidehat{(\sigma_1 \, \sigma_2)^{\ell}}$ for an integer number $\ell$.

If zero is a distinguished point, the other two distinguished points
are equidistributed on a circle with center zero. The mapping  representing  ${\mathcal H}_{\infty} ({\mathfrak m}_b)$is a
rotation, precisely, a multiplication by a power of $-1$, and
corresponds to $\reallywidehat{(\sigma_1 \, \sigma_2 \, \sigma_1)^{\ell}
/ \langle \Delta_3^2 \rangle}$ for an integer number $\ell$.
(Note that powers of $\Delta_3^2$ are reducible and the mapping class in $\mathfrak{M}(\mathbb{P}^1;\{\infty\}, E_3)$ corresponding to $\Delta_3^2$ is the identity.)
The mapping  representing  ${\mathfrak m}_b$ corresponds to  $\reallywidehat{(\sigma_1 \, \sigma_2 \, \sigma_1)^{\ell}}$.
We proved (b).

Suppose $\widehat b$ is reducible, i.e. the mapping class
${\mathfrak m}_b$ is reducible for $b \in \widehat b$.
Then by Lemma \ref{lemEl.0} the class $\widehat b$ is the conjugacy class of $\sigma_1^k\Delta_3^{2\ell}$ for integers $k\neq 0$ and $\ell$. We obtained (c). The braid $b =
\sigma_1^k \, \Delta_3^{2\ell}$ has infinite conformal module. This
gives (a) in the reducible case.

Suppose $\widehat b$ is the conjugacy class of a commutator $b \in
{\mathcal B}_3$. Then $b$ has exponent sum zero (more detailed,
the sum of exponents of the generators in a representing word equals zero). If ${\mathcal M}
(\widehat b) = \infty$ then this is possible only if $b$ is conjugate to $\sigma_1^k
\, \Delta_3^{2\ell}$
with $k+6\ell=0$. Hence, $k$ is even and the
braid $b$ is pure. This proves (d). The lemma is proved. \hfill
$\Box$

\medskip

\noindent{\bf Proof of Lemma \ref{lem8.0}.} Let  $\mathfrak{S}_f= \{(z,\zeta)\in X\times \mathbb{C}: f(z,\zeta)=0\}$ be the zero set of the algebroid function $f$. Since $f$ is irreducible, this set is connected. We obtain
an unramified covering $\mathfrak{S}_f \rightarrow X$ of degree $3$. Hence, for the Euler characteristics $\chi(X)$ and $\chi(\mathfrak{S}_f)$ the relation
$\chi(\mathfrak{S}_f)=3\chi(X)$ holds. Let $m(X)$ and $m(\mathfrak{S}_f)$ be the number of boundary components of $X$ and $\mathfrak{S}_f$. Then
$2-2g(\mathfrak{S}_f)-m(\mathfrak{S}_f)= 3(2-2g(X)-m(X))$. Hence, since $m(X)=1$, $m(\mathfrak{S}_f)$ equals either $1$ or $3$. This means that
$f\mid A$ is either irreducible or solvable. If $f\mid A$ is irreducible, it represents an irreducible conjugacy class of braids  $\widehat b_{f,A}$. By Lemma \ref{lemm8.3}  $\hat b_{f,A}$ must be the conjugacy class of a  periodic braid which corresponds to a $3$-cycle. This is impossible for conjugacy classes of products of braid commutators. \hfill $\Box$
\index{$\chi(X)$} \index{$\widehat{b}_{f,A}$}

\medskip

We will discuss in more detail the result of Gorin, Lin and Zjuzin (\cite{GL}, \cite{Z}) and give an alternative proof of a slight improvement. For simplicity we will
restrict ourselves to the case when the degree of the
quasipolynomial is a prime number. (The general case has been
considered by Lin and later by Zjuzin. An alternative proof of the
general case can be given as well.) In \cite{GL} Gorin and Lin
proved that for each prime number $n$ there exists a number $r_n$
such that any separable algebroid function of degree $n$ on an
annulus $A$ of conformal module strictly larger than $r_n$ is
reducible provided the index of its discriminant $z \to {\sf D}_n
(f_z)$, $z \in A$, is divisible by $n$. Recall that the value at a
point $z \in A$ of a separable algebroid function of degree $n$ on
$A$ is a polynomial $f_z \in {\mathfrak P}_n$. ${\sf D}_n (f_z)$
denotes its discriminant. Recall that ${\sf D}_n$ is a function on
the space of all polynomials $\overline{\mathfrak P}_n$ of degree
$n$ which vanishes exactly on the set of polynomials with multiple
zeros. Explicitly, for a polynomial ${\sf p} \in \overline{\mathfrak
P}_n$, ${\sf p} (\zeta) = \underset{j=1}{\overset{n}{\prod}} (\zeta
- \zeta_j)$, its discriminant equals
\begin{align}\label{eq8.2}
{\sf D}_n ({\sf p}) = \prod_{i < j} (\zeta_i - \zeta_j)^2 \, .
\end{align}
The index of the mapping
$$
A \ni z \to {\sf D}_n (f_z) \in {\mathbb C} \backslash \{0\} \, ,
\quad z \in A \, ,
$$
is the degree of the map
$$
z \to \frac{{\sf D}_n (f_z)}{\vert {\sf D}_n (f_z) \vert}
$$
from $\{ \vert z \vert = 1\}$ to itself. 
Zjuzin \cite{Z} proved that one can take $r_n = n \cdot \rho_0$ for an
absolute constant $\rho_0$, and
Petunin showed that $\rho_0=10^7$ works. 
We give a short conceptional proof of the theorem of these authors with an improved constant.
\begin{thm}\label{thm8.3}
Suppose
$n$ is a prime number. If $A=\{z\in \mathbb{C}:\frac{1}{\sqrt{r}}<|z|<\sqrt{r}\}$  is an annulus of conformal module
$m(A) > \frac{2\pi}{\log 2} \, n$, and $f$ is a separable algebroid function on $A$ of degree $n$ such that the index of its discriminant is divisible
by $n$, then $f$ is reducible.
\end{thm}
The present proof of the theorem uses the Main Theorem (see also Theorem 1 of \cite{Jo})
and the following result of Penner \cite{P} on the smallest non-vanishing
entropy among irreducible $n$-braids.

\begin{thm}\label{thm8.4}{\rm (Penner)} Denote by $h_{\sf{g}}^{\sf{m}}$ the
smallest non-vanishing entropy among irreducible self-homeomorphisms
of Riemann surfaces of genus $\sf{g}$ with $\sf{m}$ distinguished points
$(2{\sf{g}}-2+{\sf{m}} > 0)$. Then
$$
h_{\sf g}^{\sf m} \geq \frac{\log 2}{12{\sf g} - 12+4 {\sf m}} \, .
$$
\end{thm}
Penner's theorem implies that the smallest
non-vanishing entropy among irreducible $n$-braids, $n \geq 3$, is
bounded from below by
$$
\frac{\log 2}{4(n+1)-12} = \frac{\log 2}{4n-8} \geq \frac{\log 2}{4}
\cdot \frac 1n \, , \quad n \geq 3 \, .
$$
By the Main Theorem  the largest finite conformal module among irreducible conjugacy classes of $n$-braid does not exceed $\frac{\pi}{2} \ \frac{4}{\log2} \cdot n $ for $ n \geq 3$.

\medskip

\noindent {\bf Proof of Theorem \ref{thm8.3}} Suppose $f$ is an irreducible
separable algebroid function on the annulus $A$ and
$$
m(A) > \frac{\pi}{2} \ \frac{4}{\log2} \cdot n = \frac{2\pi}{\log 2}
\cdot n \, .
$$
By Lemma \ref{lem8.1} the conjugacy class $\widehat b_{f,A} \in \widehat{\mathcal
B}_3$ induced by $f$ is irreducible. By the Main Theorem the entropy $h
(\widehat b_{f,A})$ of the conjugacy class of braids $\widehat b_{f,A} \in
\widehat{\mathcal B}_n$ induced by $f$ is strictly smaller than
$\frac{\log2}4 \cdot \frac 1n$. By Penner's Theorem $h(\widehat b_{f,A})
= 0$, i.e. ${\mathcal M} (\widehat b_{f,A}) = \infty$.

\smallskip

This implies that $\widehat b_{f,A}$ is the conjugacy class of a
periodic braid corresponding to an $n$-cycle, hence it is the
conjugacy class of $(\sigma_1 \, \sigma_2 \ldots \sigma_n)^k$ for an
integer $k$ which is not divisible by $n$. The isotopy class of the
algebroid function $\tilde f (z,\zeta) = \zeta^n - z^k$, $z \in A$,
induces this conjugacy class of braids. Indeed, for $z = e^{2\pi i
t}$, $t \in [0,1]$, the set of solutions of the equation $\tilde f
(z,\zeta) = 0$ is
$$
E_n(t) \stackrel{def}=\left\{ e^{\frac{2\pi ik}{n} t} , e^{\frac{2\pi ik}n + \frac{2\pi i
k}n t} , \ldots , e^{\frac{2\pi i k (n-1)}n + \frac{2\pi i k}n t}
\right\} \, , \quad t \in [0,1] \, .
$$
The path $t\to E_n(t),\,t \in [0,1],$ in ${\mathfrak P}_n$ defines a geometric braid in the
conjugacy class of $(\sigma_1 \cdot \sigma_2 \cdot \ldots \cdot
\sigma_n)^k$.

\smallskip

Compute the discriminant ${\sf D}_n (\tilde f_z)$:
\begin{equation}
{\sf D}_n (\tilde f_z) \,= \,\prod_{0 \leq m < \ell < n} \left(
e^{\frac{2\pi ik}{n} \cdot m + \frac{2\pi i k}n \cdot t} -
e^{\frac{2\pi ik}{n} \cdot \ell + \frac{2\pi i k}n \cdot t}
\right)^2 \nonumber
\end{equation}
\begin{equation}\label{eq8.19} = e^{\frac{2\pi ik}{n} \cdot t \cdot n \cdot (n-1)}
\cdot \prod_{0 \leq m < \ell < n} \left( e^{\frac{2\pi ik}{n} m}
- e^{\frac{2\pi ik}{n} \ell} \right)^2 \, . 
\end{equation}
Hence
$$
{\sf D}_n (\tilde f_z) = e^{2\pi i k \cdot (n-1) \cdot t} \cdot c_n
\, , \quad z = e^{2\pi i t} \, , \quad t \in [0,1] \, ,
$$
where $c_n$ is a non-zero constant depending only on $n$. (It is
equal to the product in the last expression of \eqref{eq8.19}.) The index of
the discriminant equals $k \cdot (n-1)$ which is not divisible by
$n$. We saw that the condition for the discriminant excludes the
only possibility for a separable algebroid function on annuli of the
given conformal module to be irreducible. Hence under the conditions
of Theorem \ref{thm8.3} the algebroid function $f$ must be reducible. \hfill $\Box$

\chapter{Gromov's Oka Principle and conformal module}
\label{chapterGrom}
\setcounter{equation}{0}

While  contemporary research on the Gromov-Oka Principle has mostly been focusing on  situations when this principle holds, we study the failure and limited validity of this important principle. The obstructions to the principle are based on the relation between conformal invariants of the source and the target. The conformal invariant
used in this chapter is the collection of conformal modules of conjugacy classes of elements of the fundamental group of the manifold.

We call a mapping $f:X\to Y$ from a finite open oriented smooth surface to a complex manifold $Y$ a Gromov-Oka mapping if for each orientation preserving conformal structure on $X$
with only thick ends the mapping is homotopic to a holomorphic mapping.

We will confirm Gromov's prediction, that mappings from annuli into a complex manifold play a special role for understanding the homotopy problem of continuous mappings
to holomorphic mappings.
More precisely, it will be proved, that there are finitely many annuli in an open  oriented smooth surface $X$  of finite positive genus, that may be chosen effectively, such that a continuous mapping to the twice punctured complex plane is a Gromov-Oka mapping, if and only if its restriction to each of the annuli is a Gromov-Oka mapping. This allows a complete description of the Gromov-Oka mappings from finite open oriented smooth surfaces to the twice punctured Riemann surface. Similarly,  the Gromov-Oka mappings from a smooth oriented surface of genus   one with one hole into $\mathfrak{P}_3$ will be completely described.

\section [Gromov's Oka Principle and the conformal module]{Gromov's Oka Principle and the conformal module of conjugacy classes of elements of the fundamental group}
\label{sec:Grom.1}
Gromov  \cite{G} formulated his Oka Principle as "an expression of an optimistic expectation with regard to the validity of the $h$-principle for holomorphic maps in the situation when the source manifold is Stein". Holomorphic maps $X\to Y$ from a
complex manifold $X$ to a complex manifold $Y$ are said to satisfy the $h$-principle if each continuous map from $X$ to $Y$ is homotopic to a holomorphic map. We call a target manifold $Y$ a Gromov-Oka manifold if the $h$-principle holds for holomorphic maps from any Stein manifold to $Y$.
Gromov \cite{G} gave a sufficient condition on a complex manifold $Y$ to be a Gromov-Oka manifold.
\index{Oka Principle} \index{Gromov's Oka Principle}  \index{Gromov-Oka manifold} \index{$h$-principle} \index{Forstneric} \index{Oka manifold}

The question of understanding Gromov-Oka manifolds received a lot of attention. It turned out to be fruitful to strengthen the requirement on the target $Y$ by combining the 
$h$-principle for holomorphic maps
with a Runge type approximation property. Manifolds $Y$ satisfying the stronger requirement are called Oka manifolds. It has been proved by Forstneric that a manifold is an Oka manifold iff Runge approximation holds for holomorphic mappings from neighbourhoods of compact convex subsets of $\mathbb{C}^n$ into $Y$ (see \cite{Forst2}).
For more details, examples of Oka manifolds and an account on modern development of Oka theory based on Oka manifolds see \cite{Forst}. Notice that among the connected Riemann surfaces exactly the non-hyperbolic ones , i.e. the Riemann sphere $\mathbb{P}^1$, the complex plane $\mathbb{C}$ , the punctured complex plane $\mathbb{C}^*$, and the tori are Oka manifolds.

On the other hand, in the same paper \cite{G}  Gromov mentioned the case of maps from annuli to the twice punctured complex plane as simplest interesting example where the $h$-principle for holomorphic mappings fails.
He proposed the conformal invariants of
the twice complex plane which are called here the conformal modules of conjugacy classes of elements of the fundamental group and ascribed to them
a role in capturing certain "conformal rigidity" of the twice punctured complex plane. He supposed that for a complex manifold $Y$, which is not a Gromov-Oka manifold, the holomorphic mappings from annuli to $Y$ play a special role for understanding
obstructions to his Oka Principle.

Recall the definition of these invariants that were introduced in Chapter \ref{chapter4}
 for general complex manifolds.
Let $\mathcal{X}$ be a topological space, and $\hat e$ a conjugacy class of elements of the fundamental group of $\mathcal{X}$. $\hat e$ can be interpreted as a free homotopy class of loops in $\mathcal{X}$.
A continuous map $g$ from an
annulus $A = \{ z \in {\mathbb C} : r < \vert z \vert < R \}$ into
${\mathcal X}$ is said to represent the conjugacy class $\hat e$
if for some (and hence for any) $\rho \in (r,R)$ the map $g :
\{\vert z \vert = \rho \} \to {\mathcal X}$ represents $\hat e$.
If the mapping is a homeomorphism we will speak about an annulus contained in $\mathcal{X}$ and representing $\hat e$. If $\mathcal{X}$ is a complex manifold and the mapping $g$ is holomorphic, we will speak about a holomorphic annulus representing $\hat e$.

\begin{defn}\label{defnGrom.1}
For any complex manifold $\mathcal{X}$ and any conjugacy class $\hat e$ of elements of the fundamental group of $\mathcal{X}$ the conformal module $\mathcal{M}(\hat{e})$ of $\hat e$ is defined as the supremum of the conformal modules of annuli $\,A= \{z \in \mathbb{C}:\; r<|z|<R\}\,\,$ that admit a holomorphic mapping into $\mathcal{X}$ representing $\hat e$.
\end{defn}
Obstructions to Gromov's Oka Principle are
based on the relation between conformal invariants of the source and the target. Let $X$ and $Y$ be  topological spaces.
For an element $e\in \pi_1(X,q_0)$ we let  $\hat e$ be its free homotopy class.
Further, for a continuous mapping $f:X\to Y$ we let $f_*:\pi_1(X,q_0)\to \pi_1(Y,f(q_0)$ be the induced homomorphism and $\widehat{f_*({e})}$  the  free homotopy class of $f_*(e)$. The following observation provides an obstruction to Gromov's Oka Principle.\\
{\it Let $f:X\to Y$ be a holomorphic mapping from a complex manifold $X$ to a complex manifold $Y$. Then $\mathcal{M}(\widehat{f_*({e})})\geq \mathcal{M}(\hat{e})$ for each free homotopy class $\hat e$ of loops
in $X$.}\\
Indeed, if a holomorphic mapping $g:A\to X$ from an annulus $A$ to $X$ represents $\hat e$, then the composition $f\circ g:A\to Y$ represents $\widehat{f_*(e)}$ Hence, $\mathcal{M}(\widehat{f_*(e)})\geq m(A)$. Taking annuli of conformal module close to $\mathcal{M}(\hat{e})$, we obtain the statement.

If we know effective estimates of  conformal modules of free homotopy classes of loops in the source and in the target, we get effective information on the obstructions.

Here and in the next chapters we will study the following basic problem.
\begin{probl}\label{probl0}
Study obstructions for Gromov' Oka  Principle. In particular, look for effective estimates of the
conformal modules of conjugacy classes of elements of the fundamental group of complex manifolds,
if possible, in algebraic terms related to the fundamental group.
\end{probl}
If the source space is an annulus $A$, the conformal module of conjugacy classes of elements of the target $Y$ almost describe, which continuous mappings $A\to X$ are homotopic to holomorphic mappings.  Namely, with a generator $e$ of the fundamental group of $A$ with some base point chosen, $f$ is homotopic to a holomorphic mapping, if $m(A)<\mathcal{M}(\widehat{f_*({e})})$ and is not homotopic to a holomorphic mapping if $m(A)>\mathcal{M}(\widehat{f_*({e})})$. The case $m(A)=\mathcal{M}(\widehat{f_*({e})})$ is more subtle. For the case when the target is the twice punctured complex plane effective estimates of the conformal modules of conjugacy classes of elements of the fundamental group will be given in Chapter \ref{chapter3-braids}, so that the obstructions to Gromov's Oka Principle can be made effective in this case.

Consider the target $Y=\mathfrak{P}_n$. Recall that $\mathfrak{P}_n\cong C_n(\mathbb{C})\diagup {\mathcal S}_n$ is the space of monic polynomials without multiple zeros, parameterized either by the coefficients or by the unordered tuple of the zeros of the polynomials (see Section \ref{sec:2.1} and Chapter \ref{chapter8}). Note that $\mathfrak{P}_n=\mathbb{C}^n\setminus \{{\sf D}_n=0\}$ is the complement of an algebraic hypersurface in $\mathbb{C}^n$. Here ${\sf D}_n$ is the discriminant  (see also equality \eqref{eq8.2}).
A continuous mapping $f:A\to \mathfrak{P}_n$ represents a conjugacy class
$\widehat{b}_{f,A}$ of braids. For $n=3$ there are conjugacy classes of braids that cannot be represented by a holomorphic mapping from an annulus of conformal module bigger than
$ \frac\pi2 \left(\log \frac{3+\sqrt 5}2
\right)^{-1}$     (see Lemma \ref{lemm8.3}). The same is true for any $n>3$.
Hence,  $\mathfrak{P}_n$ is not a Gromov-Oka manifold for $n\geq 3$.

On the other hand,  
any continuous map from an oriented connected open smooth surface $X$ into $\mathfrak{P}_2$ has the Gromov-Oka property, moreover, any such map is homotopic to a holomorphic map for any conformal structure on $X$.  This statement can be proved in the same way as Statement (2) of Theorem \ref{thmEl.0} below.

By the definition of the conformal module of a conjugacy class of $n$-braids (see Definition \ref{defnGrom.1})
a continuous mapping $f$ from an annulus $A$ to the space  $\mathfrak{P}_n$ is homotopic to a holomorphic mapping if $m(A)<\mathcal{M}(\widehat{b}_{f,A})$ and is not homotopic to a holomorphic mapping if  $m(A)>\mathcal{M}(\widehat{b}_{f,A})$. In other words,
the relation of the conformal module $\mathcal{M}(\widehat{b}_{f,A})$ to the conformal invariant $m(A)$ of $A$ "almost" determines whether $f$ is homotopic to a holomorphic mapping.
Recall that $\mathcal{M}(\widehat{b}_{f,A})$ is inverse proportional to the entropy of $\widehat{b}_{f,A}$. It is reasonable to interpret the entropy as "complexity": The complexity grows with increasing entropy, while the chance for the mapping to be homotopic to a holomorphic one decreases. For $n=3$ there are effective upper and lower bounds for the conformal module (equivalently, for the entropy) of conjugacy classes of $3$-braids that differ by multiplicative constants.  We will address this topic in Chapter \ref{chapter3-braids} following \cite{Jo2}, \cite{Jo3}. Notice that also a multiple of the inverses of the conformal modules of conjugacy classes of elements of the fundamental group $\pi_1(\mathbb{C}\setminus \{-1,1\},0)$ of the twice punctured complex plane are equal to the entropies of associated mapping classes.
\index{$\widehat{b}_{f,A}$}

The problem of understanding obstructions to Gromov's Oka Principle is interesting for other target spaces $Y$ as well. In this respect the following problem occurs.
\begin{probl}\label{problGrom.0}
Study the Gromov-Oka Principle in case the target is the complement of an arbitrary algebraic hypersurface of $\mathbb{C}^n$ or the quotient of the
$n$-dimensional round complex ball by a subgroup of its automorphism
group which acts freely and properly discontinuously. Study the conformal module of conjugacy classes of elements of the fundamental group of such spaces.
\end{probl}

The collection of conformal
modules of all conjugacy classes of elements of its fundamental group $\pi_1(Y,y_0)$ which is a biholomorphic invariant of the complex
manifold $Y$ is expected to play a special role for understanding the Gromov-Oka Principle. In the case of the space $\mathfrak{P}_n$
this may be opposed to the fact, that in connection with his interest in the Thirteen's Hilbert Problem Arnold  \cite{Ar} considered the topological (cohomological) invariants of $\mathfrak{P}_n$.

The facts known for mappings from annuli to $\mathfrak{P}_n$ motivate the following two problems related to the restricted validity of Gromov's Oka Principle in general.
\begin{probl}\label{problGrom2}
Fix a connected open (i.e. non-compact) complex manifold $X$ and a complex manifold $Y$. Obtain information about the set of continuous or smooth mappings $X\to Y$ that are homotopic to holomorphic mappings.
\end{probl}

For a smooth manifold $X$ an orientation preserving homeomorphism $\omega:X\to \omega(X)$ from $X$ to a complex manifold
$\omega(X)$ is called a complex structure on $X$. If $X$ is a smooth surface and $\omega(X)$ is a Riemann surface, we will also call $\omega$ a conformal structure on $X$. A mapping $f$ from $X$ to a complex manifold $Y$ is said to be holomorphic for the complex structure $\omega$ if the mapping $f\circ \omega^{-1}$ is holomorphic on $\omega(X)$.
\index{complex structure}
\index{Oka Principle ! soft}

The following statement is known as ''soft Oka Principle'' (see \cite{Forst1}).

\smallskip

\noindent {\bf Soft Oka Principle.} {\it Any continuous mapping from a Stein manifold $X$ to a complex manifold $Y$ gets holomorphic after changing both, the mapping and the complex structure of $X$, by a homotopy.} (Forstneric, Slapar)

\smallskip

\noindent If $X$ is a finite open Riemann surface and $Y= \mathbb{C}\setminus \{-1,1\}$, the soft Oka principle is just Runge's approximation theorem on compact subsets of Riemann surfaces.
\index{Runge's Approximation Theorem} \index{Mergelyan Approximation}
Indeed, take any continuous function $f:X\to \mathbb{C}\setminus\{-1,1\}$.  Let $Q$ be a standard bouquet of circles for $X$.
The continuous function $f \mid Q$  can be approximated uniformly on $Q$ by holomorphic functions on
$X$. See \cite{Ko} for Mergelyan type approximation (approximation
of continuous functions on $Q$ by holomorphic functions in a
neighbourhood of $Q$) and \cite{Sch} for Runge approximation by
meromorphic functions on closed Riemann surfaces (here,
approximation of holomorphic functions in a neighbourhood of $Q$ by
holomorphic functions on $X$). We obtain approximating holomorphic
functions defined on $X$ which do not necessarily omit $-1$ and $1$. But their restrictions to
$Q$ are close to $f \mid Q$ , hence their restrictions to $Q$
have their image in $\mathbb{C}\setminus\{-1,1\}$
and are homotopic to $f \mid Q$ as mappings into $\mathbb{C}\setminus\{-1,1\}$. Hence an approximating
holomorphic function $F$ maps a neighbourhood $V$ of $Q$ into $\mathbb{C}\setminus\{-1,1\}$ and $F|V$ is homotopic to $f\mid V$ through mappings into  $\mathbb{C}\setminus\{-1,1\}$. We may take $V$ to be smoothly bounded
and such that there exists a continuous family of mappings $\varphi_t:X\to X, \,t\in [0,1],$ that map $X$ homeomorphically onto a domain $X_t$ in $X$ such that $\varphi_0$ is the identity and $X_1=V$.
The homeomorphisms $\varphi_t:X\to X_t$ are considered as new complex structures on $X$. The mapping
$F\circ \varphi_1$ is holomorphic on $X$ for the complex structure $\varphi_1$.

\begin{defn}\label{defn8.2} Let $X$ be an oriented finite open smooth surface and $Y$ a
complex manifold.
A continuous mapping $f:X\to Y$ is said to be homotopic to a holomorphic mapping
for a conformal structure $\omega:X\to \omega(X)$ on $X$, if $f \circ w^{-1}$ is homotopic to a holomorphic mapping from $\omega(X)$ to $Y$.
\end{defn}
The "soft Oka Principle" motivates the following problem.
\begin{probl}\label{problGrom3}
Consider an oriented connected smooth open manifold $X$, a complex manifold $Y$, and a continuous mapping $f:X\to Y$. Obtain information on the set of orientation preserving complex structures $\omega:X\to \omega(X)$ on $X$ such that $f$ is homotopic to a holomorphic mapping for $\omega$.
\end{probl}

We restrict ourselves to surfaces $X$ as source manifold. Recall the terminology introduced in Chapter \ref{chapter2}.
A surface is called finite if its fundamental group is finitely generated.
Each finite open Riemann surface $X$ is conformally equivalent
to a domain (denoted again by $X$) on a closed Riemann surface $X^c$ such that each connected component of the complement $X^c \setminus X$ is either a point or a closed topological disc with smooth boundary \cite{Sto}.
The connected components of the complement will be called holes.
A Riemann surface is called of first kind,
if it is closed, or it is obtained from a closed Riemann surface by removing finitely many points (called punctures). Otherwise the connected Riemann surface is called of second kind. If all holes of a finite open Riemann surface are closed topological discs, the Riemann surface is said to have only thick ends.

\begin{defn}\label{defn8.2'}
Let $X$ be a connected oriented finite open smooth surface, and let $Y$ be a complex manifold. We will say that a continuous mapping $f:X\to Y$ has the Gromov-Oka property if it is homotopic to a holomorphic mapping for any orientation preserving conformal structure $\omega:X\to \omega(X)$ with only thick ends (i.e. with $\omega(X)$ having only thick ends).
\end{defn}
For short we will call mappings with the Gromov-Oka property Gromov-Oka mappings.

 \index{conformal structure}

We look first at the Gromov-Oka  property for mappings from annuli into the spaces $ \mathfrak{P}_n$.
Recall that if $n=2$ then for an annulus $A$ each continuous mapping $f:A\to\mathfrak{P}_2$
has the Gromov-Oka property, moreover, it is homotopic to a holomorphic mapping for any conformal structure on $A$.
For the case $n=3$
the following lemma holds.
\begin{lemm}\label{lemGrom0}
A continuous mapping $f$ from an annulus $A=\{ r_1<|z|<r_2\}$ to $\mathfrak{P}_3$ is homotopic to a holomorphic mapping for each orientation preserving conformal structure of second kind on $A$ if and only if it has the Gromov-Oka property. This happens if and only if the inequality $\mathcal{M}(\widehat{b}_{f,A})> \frac{\pi}{2}(\log\frac{3+\sqrt{5}}{2})^{-1}$ holds.
\end{lemm}
Recall that $\log\frac{3+\sqrt{5}}{2}$ is the smallest non-vanishing entropy among $3$-braids. Recall also that
an annulus is of second kind if and only if it has at least one thick end, and a mapping $f:A\to Y$ is a Gromov-Oka  mapping if it is homotopic to a holomorphic one for any conformal structure on $A$ with only thick ends.

\medskip

\noindent{\bf Proof of Lemma \ref{lemGrom0}.} If $f$ is homotopic to a holomorphic function for any orientation preserving conformal structure of second kind on $A$ or for any conformal structure with only thick ends, then $\mathcal{M}(\widehat{b}_{f,A})=\infty$.

Suppose on the other hand that $\mathcal{M}(\widehat{b}_{f,A})> \frac{\pi}{2}\, (\log( \frac{3 +\sqrt{5}}{2}))^{-1}$.
Then Lemma \ref{lemEl.0} implies that
$\mathcal{M}(\widehat{b}_{f,A})=\infty$. Hence $f$ is homotopic to a holomorphic mapping for any orientation preserving conformal structure on $A$ of finite conformal module. There are two orientation preserving conformal structures of infinite conformal module, one conformally equivalent to the punctured disc $\mathbb{D}\setminus \{0\}$, it is of second kind, and the other conformally equivalent to the punctured complex plane, it is of first kind. By Lemma \ref{lemEl.0} a conjugacy class of braids of infinite conformal module is either a periodic braid, or it is conjugate to $\sigma_1^k$. We claim that in the case $\widehat{b}_{f,A}$ is a periodic conjugacy class of braids, there is an integer number $k$, such that $f$ is isotopic on the punctured plane to the quasipolynomial
$f_{1,k}(z,\zeta)= z^{k}-\zeta^3, \, z\in \mathbb{C}\setminus\{0\}, \,\zeta\in \mathbb{C}$,
or to $f_{2,k}(z,\zeta)= \zeta (z^k-\zeta^2), \, z\in \mathbb{C}\setminus\{0\}, \,\zeta\in \mathbb{C}$, respectively.
If $\widehat{b}_{f,A}$ is the conjugacy class of $\sigma_1^k$, we take any constant $R>1$. We claim that the mapping $f$ is isotopic on the punctured disc to $f_{3,k}(z,\zeta)= (z^k-\zeta^2) (\zeta -R)   , \, z\in \mathbb{D}\setminus\{0\}, \,\zeta\in \mathbb{C}$ for some integer number $k$.

Indeed, for $z = \frac{1}{2} e^{2\pi i
t}$, $t \in [0,1]$, the set of solutions of the equation $f_{1,k}
(z,\zeta) = 0$ is
$$
E_1^{k}(t) \stackrel{def}= 2^{-\frac{k}{3}}
e^{\frac{ 2\pi ik}{3} t}\left\{1 , e^{\frac{2\pi ik}{3} } , e^{\frac{4\pi ik}{3} }\right\}  \, , \quad t \in [0,1] \, .
$$
The path $t\to E_1^{k}(t),\,t \in [0,1],$ in ${\mathfrak P}_3$ defines a geometric braid in the
conjugacy class of $(\sigma_1 \cdot \sigma_2)^{k} $.

The set of solutions of the equation $f_{2,k}
(z,\zeta) = 0$, $z = \frac{1}{2} e^{2\pi i
t}$, $t \in [0,1]$,  is
$$
E_2^k(t) \stackrel{def}= 
2^{-\frac{k}{2}}   e^{\frac{ 2\pi ik}{2} t}\left\{-1,0,1\emph{}\right\}  \, , \quad t \in [0,1] \, .
$$
The respective path represents $\Delta_3^k$.

Finally, the set of solutions of the equation $f_{3,k}
(z,\zeta) = 0$,  $z = \frac{1}{2} e^{2\pi i
t}$, $t \in [0,1]$, is
$$
E_3^k(t) \stackrel{def}=
  \left\{2^{-\frac{k}{2}} e^{\frac{ 2\pi ik}{2} t}, - 2^{-\frac{k}{2}}e^{\frac{ 2\pi ik}{2} t} ,R\right\}  \, , \quad t \in [0,1] \, .
$$
The respective path represents $\sigma_1^k$.

The quasipolynomial $f_3$ is not
isotopic to a holomorphic one on the punctured plane. This follows from Lemma \ref{prop8.1} below.
\hfill $\Box$

\medskip

\begin{lemm}\label{prop8.1} Let $\hat b$ be a reducible conjugacy class of
$n$-braids of infinite conformal module which is not periodic.

Then there is no
holomorphic mapping from ${\mathbb C}^* = {\mathbb C} \backslash
\{0\}$ into ${\mathfrak P}_n$ representing $\hat b$.
\end{lemm}

\noindent {\bf Proof.} We have to prove that there is no holomorphic
mapping from ${\mathbb C}^*$ into ${\mathfrak P}_n$ which represents
$\hat b$. Assume the contrary. We may assume that $\hat b$ is
represented by a pure braid $b$. Indeed, if there is a holomorphic
mapping from ${\mathbb C}^*$ into ${\mathfrak P}_n$ representing
$\hat b$, then there is also such a holomorphic mapping representing
$\widehat{b^k}$ for any $k \in {\mathbb Z} \backslash \{0\}$.
The
number $k$ can be chosen so that $b^k$ is pure. By our assumption $b^k$ is not a power of  $\Delta_n^2$.
Since $b^k$ is pure, there is a lift of the mapping that represents $\widehat{b^k}$ to a mapping $g:{\mathbb C}^*\to C_n(\mathbb{C})$.
Associate to each point $z=(z_1,z_2,\ldots,z_n)\in  C_n(\mathbb{C})$ the complex affine mapping $\mathfrak{a}_z$ on $\mathbb{C}$,
$\mathfrak{a}_z(\zeta)= \frac{\zeta-z_1}{z_2-z_1} $., that maps $z_1$ to $0$ and $z_2$ to $1$.
Consider the  holomorphic mapping ${\mathbb C}^*\ni \xi \to \mathfrak{a}_{g(\xi)}(g(\xi))$ that maps a point $\xi \in {\mathbb C}^*$ to the point $\Big(\mathfrak{a}_{g(\xi)}(g_1(\xi)),   \mathfrak{a}_{g(\xi)}(g_2(\xi)),\ldots,\mathfrak{a}_{g(\xi)}(g_n(\xi))\Big)$, where $g(\xi)=\Big(g_1(\xi),g_2(\xi),\ldots,g_n(\xi)\Big)\in C_n(\mathbb{C})$.
Then for the first two coordinates of  $\mathfrak{a}_{g(\xi)}(g(\xi))$ the equalities   $(\mathfrak{a}_{g(\xi)}(g(\xi)))_1= \mathfrak{a}_{g(\xi)}(g_1(\xi))\equiv 0$ and
 $(\mathfrak{a}_{g(\xi)}(g(\xi)))_2= \mathfrak{a}_{g(\xi)}(g_2(\xi))\equiv 1$ hold.
For $j\geq 3$ the coordinate  function $ (\mathfrak{a}_{g(\xi)}(g(\xi)))_j$ is a mapping from ${\mathbb C}^*$ to $\mathbb{C}\setminus\{0,1\}$. Its lift to the universal covering
 $\mathbb{C}\cong \tilde{\mathbb C}^*$ maps the complex plane into $\mathbb{C}_+$.
Hence by Liouville's Theorem the mapping $\xi\to \mathfrak{a}_{g(\xi)}(g(\xi))$
is constant on ${\mathbb C}^*$.
Therefore, $\widehat{b^k}$ can be represented by a mapping of the form $\xi\to\mathcal{P}_{\rm sym}((\mathfrak{a}_{g(\xi)})^{-1}(c_1,\ldots,c_n)))$ for constants $c_1,\ldots,c_n$.
This means that $b^k$ is a power of  $\Delta_n^2$ in contrary to the assumption. \hfill $\Box$

\medskip

We obtained the following facts for conjugacy classes $\hat b$ of $n$-braids with infinite conformal module.
Each annulus of finite conformal module admits a holomorphic mapping
into ${\mathfrak P}_n$, representing such a class $\hat b$. The conjugacy class of each periodic braid (irreducible or not)
can be represented by a holomorphic mapping on ${\mathbb C}^*$ (the proof given in  Lemma \ref{lemGrom0} for $n=3$ works for arbitrary $n$).
The conjugacy class of reducible mappings cannot be represented by a holomorphic mapping on ${\mathbb C}^*$, but
in some cases there exists  a holomorphic
representing mapping from ${\mathbb D}^* = {\mathbb D} \backslash
\{0\}$.

\medskip

We saw that
only very few mappings $f$ with infinite $\mathcal{M}(\widehat{b}_f)$ are homotopic to a holomorphic mapping on an annulus of first kind (i.e. on the punctured plane). This fact together with Lemma \ref{lemGrom0} motivates the use of conformal structures with only thick ends in the definition of the Gromov-Oka property.

The following question arises, which is part of Problem \ref{problGrom3}.

\smallskip

\noindent {\bf Problem 8.4a.}
{\it Given a connected oriented finite open smooth surface $X$ and a complex manifold $Y$
with a set of generators whose conjugacy classes have infinite conformal module, which mappings $f:X\to Y$ have the Gromov-Oka property? In particular, which mappings $f:X\to  \mathfrak{P}_n$ have the Gromov-Oka property? Which mappings from $X$ into the $n$-punctured complex plane have the Gromov-Oka property?}

\smallskip

We will use the notion of reducible elements of the fundamental group of the $n$-punctured complex plane. The notion is an analog of the notion of reducible braids, or reducible mapping classes, respectively.

\begin{defn}\label{defnGromov2}
Let $E'$ be a finite subset of the Riemann sphere $\mathbb{P}^1$ which consists of $n+1\geq 3$ points.
Let $X$ be a connected finite open Riemann surface with non-trivial fundamental group.  
A non-contractible continuous map
$f:X \to  \mathbb{P}^1\setminus E'$ is called reducible if it is free homotopic (as a mapping to $\mathbb{P}^1\setminus E'$) to a mapping whose image is contained in $D\setminus E'$ for an open topological disc $D\subset \mathbb{P}^1$
with $E'\setminus D$
containing at least two points of $E'$.
Otherwise the mapping is called irreducible.
\end{defn}
\index{mapping to a punctured sphere ! reducible}

Let $E=\{z_1,z_2,\ldots,z_n\}$ be a subset of the complex plane that contains exactly $n$ points  and let $E'=E\cup\{\infty\}$. We assign to each continuous mapping $f:[0,1]\to \mathbb{C}\setminus E $ with $f(0)=f(1)$ the pure geometric $(n+1)$-braid $F(t)=\{f(t),z_1,z_2,\ldots,z_n\},\, t\in [0,1]$.
Approximating we may assume that the mapping $f$ is smooth. Consider a parameterizing isotopy $\varphi_t,\, t\in[0,1],$ for the geometric braid $F$ that is contained in the cylinder $[0,1]\times R\mathbb{D}$ for a large number $R$.
The equality  $\varphi_t(\{f(0),z_1,z_2,\ldots,z_n\}) =\{f(t),z_1,z_2,\ldots,z_n\},\, t\in [0,1]         $ holds.
Let $e_f$ be the element of the fundamental group $\pi_1(\mathbb{C}\setminus E, f(0))$ with base point $f(0)$, that is represented by $f$. We associate to $e_f$ the braid $b_f$, that is represented by $F$, and the
mapping class  $\mathfrak{m}_{f,\infty}\in \mathfrak{M}(\mathbb{P}^1;E\cup \{f(0),\infty\})$   that is represented by the
self-homeomorphism $\varphi_1$ of $\mathbb{P}^1$, belonging to the family $\varphi_t$ .
We may assign to the homotopy class $e_f$ of $f$ an entropy, namely the entropy
of the mapping class  $\mathfrak{m}_f$, which is a measure of the complexity of the mapping $f$.

The following lemma holds.
\begin{lemm}\label{lemGrom1'}
For a subset $E$  of the complex plane that contains exactly $n\geq 2$ points a continuous non-contractible mapping  $f:[0,1]\to \mathbb{C}\setminus E $ is reducible if and only if the associated braid $b_f$ is reducible, equivalently, if the associated mapping class $\mathfrak{m}_{f,\infty}$ is reducible.
\end{lemm}

\noindent {\bf Proof}. We may assume that the mapping $f$ is smooth and all related objects are smooth.  Suppose that the mapping $f$ is reducible.  Then there exists a free homotopy $f_s, s\in [0,1],$ of mappings from $[0,1]$ into $\mathbb{C}\setminus E$ with $f_s(0)=f_s(1)$ and  $f_0=f$, such that for all $ t\in[0,1]$ the inclusion $f_1(t)\subset D$ holds  for a disc $D$ with $(E\cup\{\infty\})\setminus D$ containing at least two points. The pure geometric braids $F_s, \, s\in [0,1],$ with varying base point, that are defined by $f_s$, are free  isotopic,  hence, represent the same conjugacy class of braids $\widehat{b}_f$.  Let $ E$ be the set $\{z_1,\ldots,z_n\}$, where the label is chosen so that  $z_1,\ldots,z_k$ are the points of $ E$ that are contained in $D$.
Let $\varphi_t,\, t\in [0,1]$ be a parameterizing isotopy for the geometric braid $\{f_1(t),z_1,\ldots,z_k\}$,
$\varphi_t(\{f_1(0),z_1,\ldots,z_k\})=\{f_1(t),z_1,\ldots,z_k\},\, t\in[0,1]$.
Since $f_1(t)\in D$ for all $t$, the parameterizing isotopy can be chosen so that all $\varphi_t$, $t\in [0,1]$, are equal to the identity on $\mathbb{C}\setminus D$.
In particular $\varphi_1 $ maps $\partial D$ onto itself. Since
$(E\cup\{\infty\})\setminus D$ contains at least two points and $E\cap D$ contains at least one point (since $f_1$ is not contractible), the mapping $\varphi_1$ of $\mathbb{P}^1$ with set of distinguished points $\{f(0),z_1,\ldots,z_n,\infty\}$ is reducible, and hence the associated mapping class
$\mathfrak{m}_{f_1}$ is reducible.  Hence ${\widehat b_f}$, and thus $b_f$, is reducible.

Vice versa, suppose  $\mathfrak{m}_f$ is reducible. Since $b_f$ is a pure braid, there exists a mapping $\varphi\in \mathfrak{m}_{f}$ that fixes each of the points $f(0),z_1,\ldots,z_n$, and fixes a closed curve $C$ pointwise. Each of the components of $\mathbb{P}^1\setminus C$ contains at least two points among the
 $f(0),z_1,\ldots,z_n,\infty$.
After applying a M\"obius transformation and perhaps relabeling the points, we may assume that $f(0) ,z_1,\ldots,z_k$ are contained in the bounded connected component $D$ of $\mathbb{C}\setminus C$, and $z_{k+1},\ldots,z_n$ are contained in $\mathbb{C}\setminus \bar{D}$. 
Take a continuous family  $\varphi_t,\, t\in[0,1],$ of self-homeomorphisms of $\mathbb{P}^1$ that fix $\mathcal{C}$ pointwise
and
join the indentity with $\varphi$.
Then the pure geometric  braid  $\mathcal{F}(t)\stackrel{def}=\varphi_t\big(\{f(0),z_1,\ldots,z_n\}\big),\,t\in[0,1],$ is isotopic (through geometric braids with fixed base point) to $\{f(t),z_1,\ldots,z_n\}, t\in[0,1],$ and the  strands of $\mathcal{F}$ with initial points $f(0), z_1,\ldots,z_k$,  are contained in the domain $D$.

We prove now, that the geometric braid
$\,\mathcal{F}(t)\,=\,\varphi_t\big(\{f(0),z_1,\ldots,z_n)\}\big)\,,\,$ $t\in[0,1],$
is isotopic  (through geometric braids  with fixed base point) to a geometric braid of the form $G(t)= \{g(t),z_1,\ldots,z_n\},$ $ t\in [0,1]$, with $g(t)\in D$ for $t\in [0,1]$.
We look at the geometric braid $\,\varphi_t\big(\{z_1,\ldots,z_n\}\big)\,,
t\in[0,1]$, that is obtained from $\mathcal{F}(t), \,t \in[0,1],$ by forgetting the strand with initial point $f(0)$. This geometric braid
 $\varphi_t\big(\{z_1,\ldots,z_n\}\big) ,\,t\in[0,1],$
is isotopic with fixed base point to the constant geometric braid $\{z_1,\ldots z_k,z_{k+1},\ldots,z_n\},
\, t\in [0,1],$ by an isotopy that fixes $\mathcal{C}$ pointwise.

By Remark \ref{rem2.2} there exists a smooth family $\psi_{t,s}$ of diffeomorphisms of $\mathbb{P}^1$
that are equal to the identity
on $\mathcal{C}$,
such that
$\psi_{t,s}\big(\{z_1,\ldots,z_n\}\big)$ realizes an isotopy with fixed base point
that joins  $\varphi_t\big(\{z_1,\ldots,z_n\}\big) ,\,t\in[0,1],$ with the constant geometric braid.
We may choose the isotopy so that for a small positive number $\varepsilon$ we have $\psi_{t,s}=\varphi_t$
for $0\leq s\leq\varepsilon$, $\psi_{t,1}={\rm Id}$, and $\psi_{0,s}={\rm Id}$. Then the equality $\psi_{t,s}\big((z_1,\ldots,z_n)\big)=(z_1,\ldots,z_n)$ holds for $(t,s)$ in the boundary of the square $[0,1]\times [0,1]$ except its left side.

We look now at  $\psi_{t,s}\big(\{f(0),z_1,\ldots,z_n)\}\big), \, t,s\in [0,1]$.
Notice that the mapping $[\varepsilon,1]\ni s\to \psi_{1,s}(f(0))$ defines a curve with base point $f(0)$ in $D\setminus \{z_1,\ldots,z_k\}$.
Let $h$ be an orientation preserving self-homeomorphism of $[0,1]\times[0,1]$ that
equals the identiy on $[0,1]\times \{0\}$ and on $\{0\}\times [0,1]$,
maps $\{1\}\times [0,1]$ onto $\{1\}\times [0,\varepsilon]$ by a homeomorphism that fixes $(1,0)$, and takes $[0,1]\times\{1\}$ homeomorphically onto
the remaining part of the boundary. The mapping $\psi_{h(t,s)}\big((f(0),z_1,\ldots,z_k)\big)$ equals
 $\varphi_t\big((f(0),z_1,\ldots,z_k)\big)$ for $t\in[0,1]$ and $s=0$, it equals $(f(0),z_1,\ldots,z_k)$ for $t=0$ and $t=1$, and
is of the form  $t\to (g(t),z_1,\ldots,z_k),$ for $s=1$. Notice that the mapping $g$ is homotopic to the inverse of the mapping  $[\varepsilon,1]\ni s\to \psi_{1,s}(f(0))$. We found an isotopy  (through geometric braids  with fixed base point) joining
$\,\mathcal{F}(t)\,=\,\varphi_t\big((f(0),z_1,\ldots,z_n)\big)\,,\,$ $t\in[0,1],$
with a geometric braid of the form $G(t)= (g(t),z_1,\ldots,z_n),$ $ t\in [0,1]$, with $g(t)\in D$ for $t\in [0,1]$.

It remains to prove the following statement.
If two pure geometric braids $(f(t),z_1,\ldots,z_n),$ $ t\in [0,1]$, and   $(g(t),z_1,\ldots,z_n),$ $ t\in [0,1]$, with $z_1,\ldots,z_n$ being distinct complex  numbers, are isotopic through geometric braids with fixed base point  $(f(0),z_1,\ldots,z_n)=(g(0),z_1,\ldots,z_n)$, then $g$ and $f$ are homotopic mappings (with fixed base point) into $\mathbb{C}\setminus\{z_1,\ldots,z_n\}=\mathbb{P}^1\setminus\{z_1,\ldots,z_n,\infty\}$. This statement is again a consequence of Remark \ref{rem2.2}. Let $f_{t,s}$ be the isotopy  through geometric braids with fixed base point (contained in the cyinder $[0,1]\times R \mathbb{D}$ for a large number $R$). 
We take the family $\varphi_{t,s}, (t,s)\in [0,1]\times[0,1]\times R\overline{\mathbb{D} }$ of Remark \ref{rem2.2},
for which $\varphi_{t,s}\big((z_1,\ldots,z_n)\big)$ is equal to the tuple of the last $n$ entries of $f_{t,s}$, and $\varphi_{0,s}=\varphi_{t,0}= \varphi_{t,1}={\rm Id} $.
The family  $ \varphi_{t,s}^{-1}(f_{t,s})$ defines an isotopy of the form $(f_s(t),z_1,\ldots,z_n)$ that joins  $(f(t),z_1,\ldots,z_n)$  and   $(g(t),z_1,\ldots,z_n)$. The family $f_s(t)$ is the required homotopy.
The lemma is proved.
\hfill $\Box$

\medskip

\index{mapping to a punctured sphere ! reducible}

We will mostly consider the case when
 $E=\{-1,1,\infty\}$. We will often refer to $\mathbb{P}^1\setminus \{-1,1,\infty\}$ as the thrice punctured Riemann sphere or the twice punctured complex plane $\mathbb{C}\setminus \{-1,1\}$. By Definition \ref{defnGromov2}
a continuous 
mapping from a Riemann surface to the twice punctured complex plane is reducible, iff it is homotopic to a mapping with image in a once punctured disc contained in $\mathbb{P}^1\setminus E$. (The puncture may be equal to $\infty$.)

To describe the reducible mappings into $\mathbb{P}^1\setminus \{-1,1,\infty\}$,
we recall, that we denoted by $a_1$ ($a_2$, respectively) the generator
of $\pi_1(\mathbb{C}\setminus\{-1,1\},0)$ that
is represented by a curve with base point $0$ that surrounds $-1$ ($1$, respectively) positively. We called $a_1$ and $a_2$ the standard generators of $\pi_1(\mathbb{C}\setminus\{-1,1\},0)$. By Definition \ref{defnGromov2}
a mapping from a finite open Riemann surface into $\mathbb{P}^1\setminus \{-1,1,\infty\}$ is reducible if and only if its monodromies are contained in a subgroup generated by a single element which is a conjugate of a power of an element of the form
\begin{equation}\label{eq8.19''}
a_1,\; a_2,\;{\mbox{ or}}\;\, (a_1a_2)^{-1}\,.
\end{equation}
Indeed, the images under the mapping of all loops on the Riemann surface represent elements of such a subgroup of $\mathbb{P}^1\setminus \{-1,1,\infty\}$ iff
the mapping is homotopic to a mapping with image in a punctured neighbourhood of $-1$  ($+1$, or $\infty$, respectively).
\begin{lemm}\label{lemGrom3} Let $X$ be a connected finite open Riemann surface with only thick ends equipped with base point $q_0$.
For each homomorphism $h:\pi_1(X,q_0)\to \pi_1(\mathbb{C}\setminus \{-1,1\},0)$, whose image is contained in a subgroup of $\pi_1(\mathbb{C}\setminus \{-1,1\},0)$  generated by a power of one of the elements of the form \eqref{eq8.19''},
there exists a holomorphic mapping $f:X\to  \mathbb{C}\setminus \{-1,1\}$, whose monodromy homomorphism $f_*:\pi_1(X,q_0)\to \pi_1(\mathbb{C}\setminus \{-1,1\},0)$ is conjugate to $h$.
\end{lemm}

\noindent {\bf Proof of Lemma \ref{lemGrom3}.}
Let $X$ be a connected finite open Riemann surface with only thick ends, and let $h:\pi_1(X,q_0)\to \Gamma$ be a homomorphism into the group $\Gamma$ generated by an element of $\pi_1(\mathbb{C}\setminus\{-1,1\},0)$ of the form \eqref{eq8.19''}. 
The Riemann surface
$X$ is conformally equivalent to a domain on a closed Riemann surface $X^c$ such that all connected components of its complement are closed topological discs. Hence, there exists an open Riemann surface $X_1\subset X^c$ such that $X$ is relatively compact in $X_1$. Let $\omega:X\to X_1$ be a homeomorphism from $X$ onto $X_1$ that is homotopic on $X$ to the inclusion $X\hookrightarrow X_1$ for which $\omega(q_0)=q_0$.
Identify the fundamental groups $\pi_1(X,q_0)$ and $\pi_1(X_1,q_0)$ by the mapping induced by $\omega$.

By Lemma \ref{lemEl2b} $X_1$ is diffeomorphic to a standard neighbourhood of a standard bouquet $B$ of circles for $X_1$. Hence, $X_1$ can be written as union $X_1=D\cup \bigcup_j V_j$ of an open disc $D$ and half-open bands $V_j$ attached to $D$ (see Section \ref{sec:2.0}).
Let $c_j$ be the circle of the bouquet that corresponds to the generator $e_j$ of the fundamental group  $\pi_1(X_1,q_0)$. We consider an open cover of $X_1$ as follows.
Put $U_0=D$. Then
$U_0$ is an open topological disc in $X_1$ that contains the base point $\omega(q_0)$. Cover each $V_j$ by two simply connected open sets $U_{j^+}$ and $U_{j^-}$
with connected and simply connected intersection, so that the  $U_{j^{\pm}}$ are disjoint from the $U_{k^{\pm}}$ for $j\neq k$, and
for each $U_{j^{\pm}}$ its intersection with $U_0$ is connected and simply connected. We may
also assume that the following holds. The intersection of three different sets among the $U_0$, $U_{j^+}$ and $U_{j^-}$ is empty. The intersections of each $c_j$ with $U_{j^+}$ and $U_{j^-}$ are connected, and each $c_j$ is disjoint from $U_{k^+} \cup U_{k^-}$  for $k\neq j$.
Label the  $U_{j^{\pm}}$ so that walking on $c_j\setminus U_0$ (which is contained in $V_j$) in the direction of the orientation of $c_j$
we first meet $U_{j^-}$.
For each $j$ the set $U_0\cup U_{j^-}\cap U_{j^+}$ is an annulus that contains $c_j$.

Define a Cousin I distribution as follows. For each $j$ we put ${\sf F}_{0,j^-}=0$ on $U_0\cap U_{j^-}$,
${\sf F}_{j^-,j^+}=0$ on $U_{j^-}\cap U_{j^+}$, and ${\sf F}_{j^+,0}=k_j$ on $U_{j^+}\cap U_0$. Since open Riemann surfaces are Stein manifolds (see e.g.  \cite{Fo}), the respective Cousin Problem has a solution  (see \cite{H1}, or Appendix \ref{ChapterA}). This means that there are holomorphic functions $F_0$ on $U_0$, $F_{j^+}$ on $U_{j^+}$ and $F_{j^-}$ on $U_{j^-}$ such that
\begin{align}\label{eq8.20}
F_0\;\;- F_{j^-}=0 \;\;\;&\mbox{on}\; U_0\;\,\cap
U_{j^-},\nonumber\\
 F_{j^-}- F_{j^+}=0\;\;\; &\mbox{on}\; U_{j^-}\cap
U_{j^+},\,\nonumber\\
F_{j^+}-F_0\;\;=k_j \;&\mbox{on}\; U_{j^+}\cap U_0\;\;.
\end{align}
The function
\begin{equation}
F=\begin{cases}
e^{2\pi i F_0} &\mbox{on}\; {U}_0,\\
e^{2\pi i F_{j^-}}& \mbox{on}\; U_{j^-}\\
e^{2\pi i F_{j^+}}&\mbox{on}\; U_{j^+} \,
\end{cases}
\end{equation}
is well-defined and holomorphic in $X_1$. Let  ${\sf e}$ be the generator of the fundamental group of $\mathbb{C}\setminus \{0\}$ with base point $F(q_0)$. The restriction of $F$ to each $c_j$ represents ${\sf e}^{k_j}$.

The restriction of $F$ to $X\Subset X_1$ is bounded. Hence, for a positive number $C$ the mapping $-1+ \frac{1}{C} F$ maps $X$ into $\mathbb{C}\setminus \{-1,1\}$ so that the monodromy along each $e_j$ equals $a_1^{k_j}$ (after identifying $\pi_1\Big(\mathbb{C}\setminus \{-1,1\},-1+ \frac{1}{C} F(q_0)\Big)$
with  $\pi_1(\mathbb{C}\setminus \{-1,1\},0)$ by a suitable isomorphism).
The case when the monodromies are powers of other elements of form \eqref{eq8.19''} is treated by composing $F$ with a suitable conformal self-mapping of the Riemann sphere that permutes the points $-1$, $1$, and $\infty$.
\hfill $\Box$

\medskip

In Section \ref{sec:Grom.2}
we will address Problem 8.4a for the case when the target manifold is the twice punctured complex plane, and in Section \ref{sec:Grom.3} for the case when $Y=\mathfrak{P}_3$ and $X$ is a smooth surface of genus one with a hole.
In Chapter \ref{chapterfin} we address Problem \ref{problGrom2}. For instance, we give an upper bound for the number of homotopy classes of irreducible mappings from a finite open Riemann surface
to the
twice punctured complex plane
that contain a holomorphic mapping.

The estimates can be regarded as a quantitative information concerning the restricted validity of Gromov's Oka Principle for some cases when the target manifold is not a Gromov-Oka manifold.

In Sections \ref{sec:9.3}-\ref{sec:9.8} we will consider the Gromov-Oka Principle for fiber bundles over a smooth torus with a hole.

\section[Gromov-Oka mappings from open Riemann surfaces to $\mathbb{C}\setminus\{-1,1\}$] {Description of Gromov-Oka mappings from open Riemann surfaces to $\mathbb{C}\setminus\{-1,1\}$}\label{sec:Grom.2}

Let ${\mathcal X}$ be a topological space and let $\hat e$
be a conjugacy class
of  elements of the fundamental group $\pi_1 ({\mathcal X} , x_0)$ of
${\mathcal X}$.
Choose a loop $\gamma$ in ${\mathcal X}$ that represents $\hat e$. Consider a complex manifold $\mathcal{Y}$ and
a continuous mapping $f:\mathcal{X}\to \mathcal{Y}$.  The conjugacy class of elements of the fundamental group of $\mathcal{Y}$ represented by the restriction $f \circ\gamma$ depends only on $f$ and on $\hat e$ and is denoted by $\hat{f_e}$.
Recall that $\mathcal{M}(\hat{f}_e)=\infty$ iff the restriction of $f$ to an annulus representing $\hat e$ has the Gromov-Oka property.

\begin{lemm}\label{lemGrom1}
Let $X$ be a connected oriented smooth finite open surface and let $e\in \pi_1(X,q_0)$ be an element whose conjugacy class can be represented
by a simple closed curve in $X$. Then for any $a>0$ there exists
an orientation preserving
conformal structure $\omega:X\to \omega(X)$ on $X$
such that the Riemann surface $\omega(X)$ contains a holomorphic annulus $A$ that represents $\hat e$ and has conformal module $m(A)>a$.
\end{lemm}
Notice that the conjugacy class of any element of a standard basis of $\pi_1(X,q_0)$ and the commutator of two elements of a standard basis of $\pi_1(X,q_0)$ corresponding to a handle can be represented by a simple closed curve.

\medskip

\noindent {\bf Proof of Lemma \ref{lemGrom1}.} Take any conformal structure $\omega':X\to \omega'(X)$ on $X$ with only thick ends. Let $\gamma$ be a smooth
simple closed curve that represents the conjugacy class $\hat e$ of $e$.
Cut $\omega'(X)$ along $\omega'(\gamma)$. Take a neighbourhood of $\omega'(\gamma)$ that is cut by $\gamma$ into connected components $X_+$ and $X_-$, each being conformally equivalent to an annulus. Take an annulus $A$ of conformal module $m(A)>a$, and glue the $X_{\pm}$ conformally to annuli $A_{\pm}$ in $A$ that
are disjoint and adjacent to different boundary components of $A$.
We obtain a Riemann surface and a homeomorphism $\omega$ from $X$ onto this Riemann surface. The mapping $\omega:X\to \omega(X)$ gives the required complex structure. \hfill $\Box$

\medskip
We will consider now the target space $Y=\mathbb{C}\setminus\{-1,1\}$.
Let $A$ be an annulus in the complex plane. By Lemmas \ref{lemm1a} and \ref{lemGrom0} a mapping $f:A\to \mathbb{C}\setminus\{-1,1\}$
has the Gromov-Oka property if and only if for a generator $e$ of the fundamental group of the annulus
the conjugacy class $\widehat{f_*(e)}$
has infinite conformal module.
By Lemmas \ref{lemm1a}, \ref{lemm8.3}, and \ref{lemGrom0}
a conjugacy class of elements of $\pi_1(\mathbb{C}\setminus\{-1,1\},0)$ has infinite conformal module iff it is represented by an integer power of one the elements \eqref{eq8.19''}. By Lemma
\ref{lemGrom1'}
this happens exactly if the mapping $f$ is reducible.

Notice that a continuous mapping $f$ from an annulus $A=\{r^{-1}<|z|<r\}$ into $\mathbb{C}\setminus\{-1,1\} $
with the Gromov-Oka property is also homotopic to a holomorphic mapping for any orientation reversing conformal structure with only thick ends on the annulus. Indeed, the mappings $z \to f(z)$ and $z\to f(\frac{1}{\bar z})$ are homotopic on $A$ since they coincide on $\{|z|=1\}$.

The following theorem concerns Problem 8.4a for smooth mappings from oriented finite open smooth surfaces $X$ of positive genus to the twice punctured complex plane and confirms the special role of the conformal module of conjugacy classes of elements of the fundamental group of the target with respect to this problem. It states the existence of finitely many annuli contained in $X$, such that a smooth orientation preserving mapping $X\to \mathbb{C}\setminus\{-1,1\}$  has the Gromov-Oka property if and only if its restriction to each of the mentioned annuli has this property.

\begin{thm}\label{thmGrom.0}
Let $X$ be a connected smooth oriented surface of positive genus $g$ with $m\geq 1$ holes.
There exists a subset $\mathcal{E}'$ of the fundamental group $\pi_1(X,q_0)$ of $X$ consisting of at most $(2g+m-1)^3$ elements such that the conjugacy class $\hat e$ of each element of $\mathcal{E}'$ can be represented by
a simple closed curve and the following holds.

A continuous
mapping $f:X\to \mathbb{C}\setminus \{-1,1\}$ has the Gromov-Oka property if and only if for each $e\in \mathcal{E}'$ the restriction of the mapping to an annulus in X that represents $\hat e$ has the Gromov-Oka property.
\end{thm}
\index{$\mathcal{E}'$}

The following theorem describes all continuous maps $X\to \mathbb{C}\setminus \{-1,1\}$ with the Gromov-Oka property. In this theorem we allow the surface $X$ to have any genus.

\begin{thm}\label{thmGrom.0'}
A continuous
mapping $f:X\to \mathbb{C}\setminus \{-1,1\}$ from a connected smooth oriented surface $X$ of genus $g\geq 0$ with $m\geq 1$ holes, $2g+m\geq 3$, into the twice punctured complex plane has the Gromov-Oka property, if and only if $f$ is either reducible, or $X$ is a smooth oriented two-sphere $S^2$ with at least three holes
and the mapping $f$ is homotopic to
a mapping that extends to an orientation preserving diffeomorphism $F:S^2\to \mathbb{P}^1$
that maps three points of $S^2$ contained in pairwise different holes to the three points $-1,1,$ and $\infty$, respectively.
\end{thm}
\index{$\mathcal{E}$} \index{$\mathcal{E}'$}
Let $\mathcal{E}$ be a standard system of generators of the fundamental group $\pi_1(X,q_0)$   of the smooth surface $X$ of genus $g$ with $m$ holes with base point $q_0\in X$. Before proving the Theorems  \ref{thmGrom.0} and \ref{thmGrom.0'}      we will describe a set $\mathcal{E}'$ of elements of $\pi_1(X,q_0)$ such that the conjugacy class of
each element of  $\mathcal{E}'$ can be represented by a simple closed curve.
Later we will compare the Gromov-Oka property of
a mapping $X\to \mathbb{C}\setminus \{-1,1\}$
with the Gromov-Oka property of its restriction to each annulus in $X$
that represents $\hat e$ for an element $e\in \mathcal{E}'$.

\medskip

\noindent {\bf Let first} $g>0$. For the $j$-th handle we choose three elements
$e_{2j-1}$, $e_{2j}$, and $[e_{2j-1},e_{2j}]$. The conjugacy class of each of them can be represented by a simple closed curve. For the $\ell$-th hole, $\ell=1,\ldots,m-1,$ we take the element $e_{2j+\ell}$. There is a simple closed curve that represents its conjugacy class. The obtained $3g+m-1$ elements of the fundamental group will be contained in $\mathcal{E}'$.

We convert each unordered pair of different elements of $\mathcal{E}$, that is different from a pair $\{e_{2j-1}$, $e_{2j}\}$ corresponding to a handle, into an ordered pair $(e',e'')$.
There is a simple closed curve representing $\reallywidehat{e'\,e''}=\reallywidehat{e''\,e'}$.
The products $e' e''$
constitute a set of
$\frac{1}{2}(2g+m-1)(2g+m-2) - g$
elements of the fundamental group. They will be contained in $\mathcal{E}'$.

For each pair $e_{2j-1},e_{2j}$ of elements of $\mathcal{E}$ that corresponds to a handle, and each other element $e'$ among the chosen generators (if there is any) we consider the elements $e_{2j-1}^2\,e_{2j}\, e'$, $e_{2j-1}^3\,e_{2j}\, e'$,
$e_{2j-1}\,e_{2j}^2\, e'$, and $e_{2j-1}\,e_{2j}^3\, e'$. For the conjugacy class of each such element there is a simple closed curve in $X$ representing it.
The
obtained $4g\, (2g+m-3)$ elements
of the fundamental group will be contained in $\mathcal{E}'$.

Suppose $m>2$. Let $(e_1,e_2)$ be the pair of elements of $\mathcal{E}$ that corresponds to the first handle. We convert each unordered pair of different elements of $\mathcal{E}$
corresponding to holes to an ordered pair $(e',e'')$, so
that the conjugacy class of the element $e'e_1 e' e_2 e''$ can be represented by a simple closed curve in $X$. (See Figure \ref{FigGrom1}.) We obtain no more than 
$\frac{1}{2}(m-1)(m-2)$ elements of $\pi_1(X,q_0)$.
They will be contained in $\mathcal{E}'$.

\begin{figure}[H]
\begin{center}
\includegraphics[width=9cm]{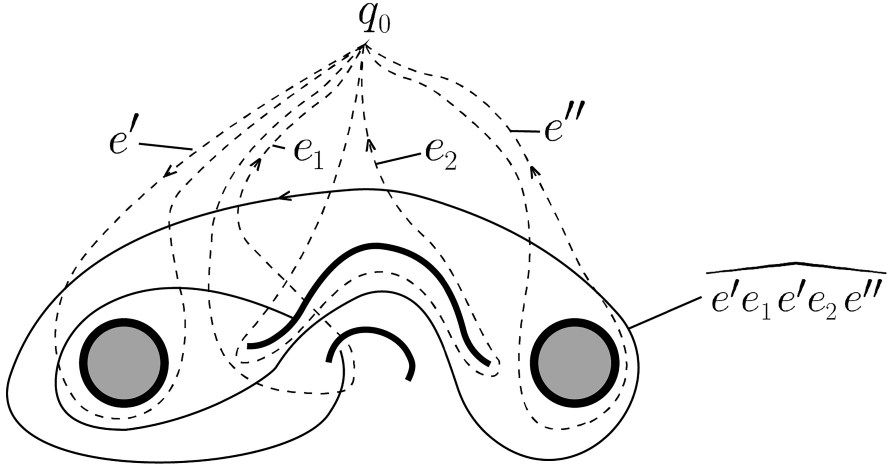}
\end{center}
\caption{A simple closed curve that represents the free homotopy class of $e'e_1 e' e_2 e''$}\label{FigGrom1}
\end{figure}

\noindent {\bf Suppose now} $g(X)=0$. If $m=1$ the fundamental group is trivial. If $m=2$ it has one generator. The set $\mathcal{E}'$ will be equal to the set $\mathcal{E}$ consisting of the generator. Let $m\geq 3$. The element $e_m\stackrel{def}=(\prod _{j=1}^{m-1}e_j)^{-1}$ is represented by a simple closed curve that surrounds $\mathcal{C}_m$ positively (i.e. walking along the curve the set $\mathcal{C}_m$ is on the left).
All elements $e\in \mathcal{E}\cup \{e_m\}$ will be contained in $\mathcal{E}'$. Further,
each unordered pair of different elements of $\mathcal{E}\cup \{e_m\}$ is converted into an ordered pair $(e',e'')$ and
the product $e'e''$ will be contained in $\mathcal{E}'$. Finally,
we convert each unordered triple of different elements of $\mathcal{E}\cup \{e_m\}$ into an ordered triple $(e',e'',e''')$ so that $\reallywidehat{e'\,e''\,e'''}$ can be represented by a simple closed curve.
The product $e'e''e'''$
will be contained in $\mathcal{E}'$.

We described all elements of $\mathcal{E}'$. Recall that the conjugacy class of each of them can be represented by a simple closed curve. We estimate now the total number $N$ of elements of $\mathcal{E}'$.

Let first $g>0$. With $x=2g+m-1\geq 2$ the number $N$ does not exceed
\begin{equation*}
\begin{cases}
x+g+ \frac{1}{2} x(x-1)-g + 4g (x-2), & 0< m\leq 2,\\
x+g+ \frac{1}{2} x(x-1)-g+  4g (x-2)+ \frac{1}{2}(m-1)(m-2),& \quad \;\;\; m>2.\\
\end{cases}
\end{equation*}

Since $2g\leq x$, for $m\leq 2$ the sum
does not exceed $x+x(x-1) +2x(x-2)$.
Hence, $N\leq 3 x^2 - 4x < x^3$.
For $m>2$ we use the inequality $m-1= x-2g \leq x-2$, hence $\big(4g+\frac{1}{2}(m-2)\big)(x-2)\leq 2 x(x-2)$.
Again $N\leq x^3$.

Let $g=0$. If $m=2$ the set $\mathcal{E}'$ contains a single element, and $N=1=(m-1)^3=x^3$.
If $x=m-1\geq 2$ the number $N$ is not bigger than $x+1 + \frac{1}{2}(x+1)x + \frac{1}{6}(x+1)x(x-1)= \frac{1}{2}(x+1)(x+2)+ \frac{1}{6} x(x^2-1)=\frac{1}{6}x^3 +\frac{1}{2}x^2 +(\frac{3}{2}-\frac{1}{6})x+1$ and since $x\geq 2$, also in this case $N\leq x^3$.
We estimated
the number of elements of $\mathcal{E}'$
from above by $x^3=(2g+m-1)^3$.

\medskip

For the proof of Theorem \ref{thmGrom.0} we need the following lemma. Its proof will given after the proof of the theorems.

\begin{lemm}\label{lemGrom2}
Let $X$ be a connected oriented smooth open surface of genus $g$ with $m$ holes, $2g+m\geq 0$, and $\mathcal{E}'$ the set of elements of the fundamental group $\pi_1(X,q_0)$ chosen above. Assume the restriction of a continuous mapping $f:X\to \mathbb{C}\setminus \{-1,1\}$ to each annulus in $X$ representing the conjugacy class $\hat e$ of an element $e\in \mathcal{E'}$ has the Gromov-Oka property.

\noindent Then one of the following statements holds.\\
$1.$ The image of the monodromy homomorphism $f_*:\pi_1(X,q_0)\to \pi_1(\mathbb{C}\setminus \{-1,1\},f(q_0))\cong
\pi_1(\mathbb{C}\setminus \{-1,1\},0)$ is contained in the group generated by a conjugate of a power of one of the elements of the form \eqref{eq8.19''}. In this case the mapping is reducible.\\
\noindent $2.$ $X$ is equal to the smooth oriented $2$-sphere denoted by $S^2$
with at least three holes removed and the mapping $f$ is homotopic to
a mapping that extends to a
diffeomorphism $F:S^2\to \mathbb{P}^1$
that maps three points of $S^2$ contained in pairwise different holes to the three points $-1,1,$ and $\infty$, respectively. \\
If $F$ is orientation preserving, than $F$ is homotopic to a holomorphic map for any orientation preserving conformal structure on $X$, including structures of first kind. Moreover, the restriction $F\mid X$ is irreducible.\\
If $F$ is orientation reversing, then for any orientation preserving conformal structure on $X$ the mapping
$F$ is homotopic to an antiholomorphic mapping, but $F$ is not a Gromov-Oka mapping.
\end{lemm}

\medskip

\noindent {\bf Proof of Theorem \ref{thmGrom.0}}.
If a mapping $f:X\to \mathbb{C}\setminus \{-1,1\}$ has the Gromov-Oka property then by Lemma \ref{lemGrom1} the restriction of $f$ to each annulus in $X$ representing the conjugacy class of an element of $\mathcal{E}'$ has the Gromov-Oka property.

Vice versa, suppose $X$ has positive genus and the restriction of a continuous mapping $f:X\to \mathbb{C}\setminus \{-1,1\}$ to each annulus in $X$ representing the conjugacy class $\hat e$ of an element $e\in \mathcal{E'}$ has the Gromov-Oka property. Then by Lemma \ref{lemGrom2} the monodromy homomorphism $f_*$
is contained in the group generated by a conjugate of a power of an element of the form \eqref{eq8.19''}, equivalently $f$ is reducible.
By Lemma \ref{lemGrom3}
each reducible map
$X\to \mathbb{C}\setminus\{-1,1\}$
is homotopic to a holomorphic mapping for each conformal structure on $X$ with only thick ends. Hence, the mapping $f$ has the Gromov-Oka property.
\hfill $\Box$

\medskip

\noindent{\bf Proof of Theorem \ref{thmGrom.0'}}.  Suppose first that $g>0$.
If a continuous mapping $f:X\to \mathbb{C}\setminus\{-1,1\}$ has the Gromov-Oka property,
then by Lemma \ref{lemGrom1} the restriction of $f$ to each annulus represented by a conjugacy class $\hat e$ for an element $e\in \mathcal{E'}$ has the Gromov-Oka property. By Lemma \ref{lemGrom2} the monodromy homomorphism $f_*$
is contained in the group generated by a conjugate of a power of an element of the form \eqref{eq8.19''} (equivalently $f$ is reducible).

Vice versa, let $X$ be an oriented smooth finite surface of any genus. By Lemma \ref{lemGrom3}
each reducible map $f:X\to \mathbb{C}\setminus\{-1,1\}$  is homotopic to a holomorphic mapping for each conformal structure on $X$ with only thick ends, i.e. reducible maps have the Gromov-Oka property.

Let $g=0$. Then $X$ is a smooth oriented $2$-sphere denoted by $S$ with at least $3$ holes removed. If an irreducible continuous mapping $f:X\to \mathbb{C}\setminus\{-1,1\}$ has the Gromov-Oka property, then by Lemmas \ref{lemGrom1} and  \ref{lemGrom2} the mapping $f$ is homotopic to a mapping that extends to
a diffeomorphism $F:S\to \mathbb{P}^1$ that maps three points contained in different holes to the points $-1$, $1$, and $\infty$. If $F$ is orientation preserving then for any conformal structure on $S$ we need a homotopy to a conformal mapping from $\mathbb{P}^1$
onto itself that takes an ordered tuple of three distinct points in  $\mathbb{P}^1$ to the tuple $(-1,1,\infty)$. This is always possible.

If the mapping $F$ is not orientation preserving,  $F$ is not homotopic to a holomorphic mapping by Lemma \ref{lemGrom2}. \hfill $\Box$

The theorem shows that for oriented finite open surfaces of positive genus only the reducible mappings $X\to \mathbb{C}\setminus\{-1,1\}$ have the Gromov-Oka property.
In case $X$ is the oriented two-sphere with at least three holes there are also irreducible homotopy classes of mappings with the Gromov-Oka property. Each consists of mappings with the following property. For each orientation preserving homeomorphism $\omega:X\to \omega(X)$ onto a Riemann surface $\omega(X)$ (maybe, of first kind) $f\circ \omega^{-1}$ is homotopic to a holomorphic mapping that extends to a conformal mapping $\mathbb{P}^1\to \mathbb{P}^1$ that maps three points in different holes to $-1,1$ and $\infty$, respectively (in some order depending on the class). These homotopy classes are the only irreducible homotopy classes with the Gromov-Oka property, and they are the only homotopy classes of mappings $X\to \mathbb{C}\setminus \{-1,1\}$ that contain a holomorphic mapping for any conformal structure on $X$ including conformal structures of first kind. Indeed,
if the image of a non-trivial monodromy homomorphism $f_*$ is generated by a conjugate of a power of one of the elements of the form \eqref{eq8.19''} the mapping $f$ cannot be homotopic to a holomorphic mapping for a conformal structure of first kind
(see the proof of Corollary \ref{corfin1a}).

\medskip

\noindent {\bf Proof of Lemma \ref{lemGrom2}.}\\
\noindent{\bf 8.6.1.} Identifying the fundamental groups $\pi_1(\mathbb{C}\setminus\{-1,1\},0)$ and $\pi_1(\mathbb{C}\setminus\{-1,1\},f(q_0))$ by a fixed isomorphism, we obtain for each $e\in\mathcal{E}'$ the equation $f_*(e)=w_e^{-1}\,b_e\,w_e$ for an element
$w_e\in \pi_1(\mathbb{C}\setminus\{-1,1\},0)$ and an element $b_e$ that is a power of an element of the form \eqref{eq8.19''}.
The images $f_*(e_{2j-1})$ and $f_*(e_{2j})$ of two elements corresponding to a handle commute. This is a particular case of the situation treated in the proof of Theorem \ref{thmGrom2} below, but can be seen easily directly.
Indeed, the sum of exponents of the terms of a word representing a commutator is equal to zero.
Since the commutator $[f_*(e_{2j-1}),f_*(e_{2j})]=f([e_{2j-1},e_{2j-1}])$ must be a power of a conjugate of an element of the form \eqref{eq8.19''} it must be equal to the identity.

Recall that a word in the generators of a free group is called reduced, if neighbouring terms are powers of different generators. We will identify elements of a free group (in particular of  $\pi_1(\mathbb{C}\setminus\{-1,1\},0)$) with reduced words in generators of the group. A word is called cyclically reduced, if either the word consists of a single term, or it has at least two terms and the first and the last term of the word are powers of different generators.

\noindent {\bf 8.6.2.} Suppose for two generators $e^{(1)}, e^{(2)} \in \mathcal{E}\subset \mathcal{E}'$ we have
 $f_*(e^{(\ell)})=w_{e^{(\ell)}}^{-1} b_{e^{(\ell)}} w_{e^{(\ell)}}$ with
$b_{e^{(\ell)}}=a_{j_{\ell}}^{k_{\ell}}$, $\ell=1,2,$ where each $a_{j_{\ell}}$ is one of the generators $a_1$, $a_2$ of $\pi_1(\mathbb{C}\setminus\{-1,1\},0)$. Then there exists an element $w_{{e^{(1)},{e^{(2)}}}}$ of $\pi_1(\mathbb{C}\setminus\{-1,1\},0)$ which conjugates both  $f_*(e^{(\ell)}) $ to
$a_{j_{\ell}}^{k_{\ell}}$, $\ell=1,2,$.
Indeed,
put $w={w_{e^{(2)}}}w_{e^{(1)}}^{-1}$. Since either $e^{(1)}e^{(2)}\in \mathcal{E}'$ or  $e^{(2)}e^{(1)}\in \mathcal{E}'$, the monodromy  $f_*(e^{(1)}e^{(2)})=f_*(e^{(1)})f_*(e^{(2)})$ has infinite conformal module and is conjugate to $a_{j_1}^{k_1} w^{-1} a_{j_2}^{k_2} w$. The element $a_{j_1}^{k_1} w^{-1} a_{j_2}^{k_2} w$ is conjugate to $a_{j_1}^{k_1} {w'}^{-1} a_{j_2}^{k_2} w'$ for an element $w'= a_{j_2}^{-k'_2}w a_{j_1}^{-k'_1}
\in \pi_1(\mathbb{C}\setminus\{-1,1\},0)$ that is either equal to the identity or
it can be written as a reduced word that starts with a power of $a_{j_1}$ and ends with a power of $a_{j_2}$. (In the case $a_{j_1}= a_{j_2}$ we allow that $w'$ is equal to a power of the other generator of $\pi_1(\mathbb{C}\setminus\{-1,1\},0)$.)
If $w'$ is not the identity, then $a_{j_1}^{k_1} {w'}^{-1} a_{j_2}^{k_2} w'$ can be written as cyclically reduced word that contains powers of different sign of generators of $\pi_1(\mathbb{C}\setminus\{-1,1\},0)$. Any cyclically reduced word that represents a conjugate of a power of an element of the form \eqref{eq8.19''} contains only powers of equal sign of the generators. Hence $w'={\rm Id}$, and $w=a_{j_2}^{k'_2} a_{j_1}^{k'_1}$, equivalently,
\begin{align}\label{eqGrom1}
a_{j_2}^{-k'_2} w_{e^{(2)}} = a_{j_1}^{k'_1} w_{e^{(1)}}\,.
\end{align}
Put $w_{{e^{(1)},{e^{(2)}}}}= a_{j_2}^{-k'_2} w_{e^{(2)}} = a_{j_1}^{k'_1} w_{e^{(1)}} $. Then $w_{e^{(1)},e^{(2)}} f_*(e^{(1)}) (w_{e^{(1)}e^{(2)}})^{-1}= a_{j_1}^{k_1}$ and $w_{e^{(1)},e^{(2)}} f_*(e^{(2)}) (w_{e^{(1)},e^{(2)}})^{-1}= a_{j_2}^{k_2}$.

\noindent {\bf 8.6.3.} As a corollary we see that $f_*(e^{(1)}e^{(2)})$ is conjugate to $a_{j_1}^{k_1} a_{j_2}^{k_2}$. Hence, either at least one of the monodromies $f_*(e^{(\ell)})$, $\ell=1,2,$ is the identity, or $j_1=j_2$, or $k_1=k_2=\pm 1$.

Further, if the monodromies $f_*(e^{(\ell)})$, $\ell=1,\ldots, k,$
along a collection of elements of $\mathcal{E}$ are conjugate to powers of a common single generator $a_j$,
there is a single element of $\pi_1(\mathbb{C}\setminus\{-1,1\},0)$  that conjugates the monodromy along each element of this collection to a power of $a_j$. Indeed,
suppose $w_{e^{(\ell)}}$ conjugates the monodromy $f_*(e^{(\ell)})$ to a power of $a_j$. Then equality \eqref{eqGrom1} implies
to each pair $e^{(1)},\,e^{(\ell)}, \, \ell=2,\ldots,k,$ of elements of this collection. By this equation $w_{e^{(\ell)}}=a_j^{k'_{\ell}}w_{e^{(1)}}$
and, hence $w_{e^{(1)}}^{-1}$ conjugates each $f_*(e^{(\ell)})$ to a power of $a_j$.

Moreover, if each monodromy of a collection $f_*(e^{(1)}),\ldots,         f_*(e^{(k)})$, $e^{(\ell)}\in \mathcal{E}, \ell =1,\ldots,k,$ is conjugate to a power of a common element of form \eqref{eq8.19''} and the monodromy $f_*(e^{(k+1)})$, $e^{(k+1)}\in \mathcal{E},$  is conjugte to a different element of form \eqref{eq8.19''}, then there exists a single element $w\in \pi_1(\mathbb{C}\setminus\{-1,1\},0)$, that conjugates
all these monodromies to powers of  elements of form \eqref{eq8.19''}.
Indeed, such a $w$ exists for the pair $f_*(e^{(1)})$ and $f_*(e^{(k+1)})$, and by the preceding arguments $w$ also conjugates $f_*(e^{(2)}),\ldots,         f_*(e^{(k)})$ to  powers of  elements of form \eqref{eq8.19''}.

\noindent {\bf 8.6.4.} The same arguments apply if for two elements $e^{(1)},e^{(2)}\in \mathcal{E}$ the monodromy $f_*(e^{(1)})$ is conjugate to a power of $a_j$ and the monodromy $f_*(e^{(2)})$ is conjugate to a power of $a_1a_2$. We replace the generators $a_1$, $a_2$ of the free group $\pi_1(\mathbb{C}\setminus\{-1,1\},0)$ by the generators
$A_1=a_j$ and $A_2=(a_1a_2)^{-1}$. Notice that $A_2$ can be considered as element of the fundamental group of the thrice punctured plane that is represented by loops surrounding $\infty$ positively.

\noindent The arguments above imply the following.\\
{\it 
If the monodromy $f_*(\tilde{e})=w^{-1}v^k w$ along an element $\tilde{e}\in \mathcal{E}$ is a conjugate to the $k$-th power of an element $v$ among  $a_1$, $a_2$, $a_1a_2$,
with $|k|>1$, then for any $e'\in \mathcal{E}$, $e'\neq \tilde e$, the monodromy $f_*(e')$ equals $w^{-1}v^{k'} w$ for an integer $k'$.}

\noindent {\bf 8.6.5.} We consider now the case when there is no element in $\mathcal{E}$ the monodromy along which is conjugate to the $k$-th power of an element of form \eqref{eq8.19''} with $|k|>1$.
We claim that if $g>0$ and the monodromy along
an element $e_{2j} \in \mathcal{E},\, j\leq g,$ (i.e. $e_{2j}$ belongs to a pair corresponding to a handle) is not trivial and, hence, is equal to a conjugate $w^{-1} v^{\pm 1} w$ for an element $v$ of form \eqref{eq8.19''},
then
all monodromies are powers of the same element $w^{-1} v w$ or   $w^{-1} v^{- 1} w$, respectively.
Indeed, by our assumption all non-trivial monodromies along elements $e$ of $\mathcal{E}$ are conjugate to $v_e^{\pm 1}$ for an element $v_e$ of the form \eqref{eq8.19''}.
Consider the element $e_{2j-1}$ for which the pair $e_{2j-1}, e_{2j}\in \mathcal{E}$ corresponds to a handle. Conjugating all monodromies by a single element, we suppose for instance that $f_*(e_{2j})=a_1$. Since the monodromies corresponding to a pair of handles commute (see part 8.6.1 of the proof), the monodromy $f_*(e_{2j-1})$ is a power of $a_1$, hence, by our assumption $f_*(e_{2j-1})$ is equal to $a_1^{\pm 1}$, or it is the identity.
Since for an element $e'\in \mathcal{E}$ different from $e_{2j-1}, e_{2j}$ the products $e_{2j-1}\, e_{2j}^2\,e'$ and $e_{2j-1}\, e_{2j}^3\,e'$ are contained in $ \mathcal{E}'$, the monodromies $f_*(e_{2j-1}\, e_{2j}^2\,e')$ and $f_*(e_{2j-1}\, e_{2j}^3\,e')$ must be  be conjugate to a power of an element of the form \eqref{eq8.19''}.
If $f_*(e_{2j-1})$ equals $a_1$ or the identity, then
the monodromy
$f_*(e_{2j-1}\, e_{2j}^2\,e')$ can only be conjugate to a power of an element of the form \eqref{eq8.19''}, if $f_*(e')$ is a power of $a_1$.
If $f_*(e_{2j-1})$ is equal to $a_1^{-1}$, then
$f_*(e_{2j-1}\, e_{2j}^3\,e')$ can only be conjugate to a power of an element of the form \eqref{eq8.19''}, if $f_*(e')$ is a power of $a_1$.

The case when $f_*(e_{2j})$ equals $a_1^{-1}$, $a_2^{\pm 1}$, or $(a_1 a_2)^{\pm 1}$,
or when the role of $e_{2j-1}$ and $e_{2j}$ is interchanged, can be treated similarly.
The claim is proved.

{\bf 8.6.6.} Suppose $g(X)>0$ but all monodromies along elements of $\mathcal{E}$ corresponding to handles are trivial.
We claim that still all monodromies are powers of the same conjugate of an element of form \eqref{eq8.19''}. If the monodromies along all but possibly one element of $\mathcal{E}$ are trivial, there is nothing to prove. Hence, we may assume that $m>2$.
By our assumption all non-trivial monodromies along elements of $\mathcal{E}$ are conjugate to an element of the form \eqref{eq8.19''} or are inverse to a conjugate of an element of the form \eqref{eq8.19''}. Let $e_1$ and $e_2$ be the elements of $\mathcal{E}$ corresponding to the first labeled handle.
Suppose there is an unordered pair $\{e',e''\}$ of different elements of $\mathcal{E}$  so that the monodromy along each element of the pair is non-trivial. Note that  none of the elements equals
$e_1$ or $e_2$. Convert the unordered pair into an ordered pair $(e',e'')$, so that the conjugacy class $\reallywidehat{e'e_1e'e_2e''}$ can be represented by a simple closed curve.
Assume that $f_*(e')=a_1$.
If $f_*(e'')$ is not a power of $a_1$ then
$f_*(e'e_1e'e_2e'')= f_*(e')^2 f_*(e'')$ (see Figure \ref{FigGrom1})
cannot be a power of an element of the form \eqref{eq8.19''}. The remaining cases in which $f_*(e')$ is of the form \eqref{eq8.19''} or is inverse to an element of the form \eqref{eq8.19''} are treated in the same way. The claim is proved.

\noindent{\bf 8.6.7.}
Let $g(X)=0$, i.e. $X$ equals the smooth  oriented $2$-sphere $S^2$ with holes.
Suppose $f:X\to \mathbb{C}\setminus\{-1,1\}$ is a continuous mapping whose restriction to each annulus representing an element $e\in\mathcal{E}'$ has the Gromov-Oka property.
Assume that the monodromies are not all conjugate to powers of a single common element of the form \eqref{eq8.19''}. Then by our assumption
no monodromy is conjugate to the $k$-th power of an element of the form  \eqref{eq8.19''} with $|k|>1$, and there are at least three holes, and
at least two elements $e_{j'}$ and $e_{j''}$ in $\mathcal{E}$ with monodromies $f_*(e_{j'})$ and $f_*(e_{j''})$ being non-trivial powers of absolute value $1$ of conjugates of different elements of the form \eqref{eq8.19''}.

Recall that the element $e_m= (\prod_{j=1}^{m-1} e_j)^{-1} \in \pi_1(X,q_0)$ is represented by a loop with base point $q_0$ that surrounds the last hole $\mathcal{C}_m$
counterclockwise. The product $\prod _{j=1}^m f_*(e_j)$ is equal to the identity.
The product of two non-trivial powers of different elements of the form \eqref{eq8.19''}
cannot be equal to the identity. Hence, by part 8.6.3. of the proof
there must be an integer number $j''' \in [1,m]$ different from $j'$ and $j''$,
for which $f_*(e_{j'''})\neq {\rm Id}$.

The monodromies along two different elements from $\mathcal{E}\cup\{e_m\}$ cannot be non-trivial powers of the same element of the form \eqref{eq8.19''}.
Otherwise by part 8.6.3. of the proof there would be an ordered couple or an ordered triple of different elements of $\mathcal{E}\cup\{e_m\}$ whose product is in $\mathcal{\mathcal{E}'}$  but the monodromy along the product cannot be a conjugate of a power of an element of the form \eqref{eq8.19''}.
Indeed, assume for instance that two monodromies $f_*(e')$ and $f_*(e'')$  and are conjugate to powers of $a_1$ and a third momodromy $f_*(e''')$
is conjugate to a power of $a_2$.  By part 8.6.3. of the proof we may assume that $f_*(e')=a_1^{k'}$, $f_*(e'')=a_1^{k''}$, and $f_*(e''')=a_2^{k'''}$ with $|k'|=|k''|=|k'''|=1$. If $k'$ and $k''$ have different sign, then one of the monodromies $f_*(e' e''')$ or $f_*(e'' e''')$ cannot be conjugate to a power of an element of form \eqref{eq8.19''}. If  $k'$ and $k''$ have equal sign, the monodromy along the product of the three elements $(e',e'',e''')$ in any order
is not a power of an element of form \eqref{eq8.19''}. But the product of the three elements $(e',e'',e''')$ for some order is an element of $\mathcal{E}'$.
For other possible combinations of powers of elements of  form \eqref{eq8.19''} the arguments are the same.

\noindent {\bf 8.6.8.} Part 8.6.7 of the proof shows that the monodromy along at most three elements of $\mathcal{E}\cup\{e_m\}$ is nontrivial.
Hence, there are exactly three elements $e_{j'}$, $e_{j''}$, and $e_{j'''}$ among the $e_k,\, k=1,\ldots,m$, with non-trivial monodromy.
After conjugating all monodromies by a single element of $\pi_1(\mathbb{C}\setminus\{-1,1\})$  the monodromies along two of them are equal to either $a_1$ and $a_2$, respectively, or to $a_1^{-1}$ and $a_2^{-1}$, respectively. (The combinations $a_1$,  $a_2^{-1}$ or $a_1^{-1}$,  $a_2$ are impossible.)
Order the three elements by $(e',e'',e''')$ so that the product is in $\mathcal{E}'$.
After a cyclic permutation which does not change the conjugacy class of the product, we may assume that the monodromy along $e'''$ is not equal to a power of an $a_j$. Then the ordered triple of monodromies along $(e',e'',e''')$ is either $(a_1,a_2,(a_1a_2)^{-1})$, or $(a_2,a_1,(a_2a_1)^{-1})$, or $(a_1^{-1},a_2^{-1},a_2a_1)$, or $(a_2^{-1},a_1^{-1},a_1 a_2)$. The elements $a_1\,,a_2$, and $(a_1a_2)^{-1}$, respectively, of the fundamental group of the twice punctured complex plane are represented by curves that surround  positively the points $-1,\,1$, and $\infty$, respectively, the elements
$a_1^{-1},a_2^{-1},a_2a_1$ are represented by curves that surround $-1,1,\infty$ negatively, and similarly for the remaining cases.

The case $(a_1,a_2,(a_1a_2)^{-1})$
for the monodromies along $(e',e'',e''')$  corresponds to the homotopy class that contains the following mapping. Let $\mathcal{C}',\mathcal{C}'',\mathcal{C}'''$ be the holes of $X$ corresponding to $e',e'', e'''$, respectively.
Take points $p'\in \mathcal{C}',\, p''\in \mathcal{C}'',$ and $p'''\in \mathcal{C}'''$, respectively. Denote by $F$ an orientation preserving
diffeomorphism from $S^2$ onto $\mathbb{P}^1$,
that maps  $p'$ to $-1$, $p''$ to $1$, and $p'''$ to $\infty$. It is straightforward to check that $F$ has the required monodromies along all generators. Hence, $f$ is homotopic to $F\mid X$. The case 
$(a_2,a_1,(a_2a_1)^{-1})$
for the monodromies along $(e',e'',e''')$ is similar. In these two cases
for any orientation preserving conformal structure $\omega:X\to \omega(X)$, including conformal structures of first kind, the mapping $(F|X)\circ{\omega}^{-1}$ is homotopic to a
mapping hat extends to $\mathbb{P}^1=\omega(X)^c$ as a conformal self-diffeomorphism of $\mathbb{P}^1$ that maps $X$ into $\mathbb{C}\setminus\{-1,1\}$.

If $f$ has monodromies $(a_1^{-1},a_2^{-1},a_2a_1)$ along $(e',e'',e''')$,
then $f$ is homotopic to $F\mid X$ for an orientation reversing
diffeomorphism $F$ from $S^2$ onto $\mathbb{P}^1$,
that maps  $p'$ to $-1$, $p''$ to $1$, and $p'''$ to $\infty$. The case $(a_2^{-1},a_1^{-1}, a_1a_2)$ is similar.
In these cases for any orientation preserving conformal structure $\omega:X\to \omega(X)$, including conformal structures of first kind, the mapping $(F|X)\circ{\omega}^{-1}$ is homotopic to a
mapping hat extends to $\mathbb{P}^1=\omega(X)^c$ as an anti-conformal self-homeomorphism of $\mathbb{P}^1$ that maps $X$ into $\mathbb{C}\setminus\{-1,1\}$. But the mapping $f$ does not have the Gromov-Oka property.

The latter fact can be seen as follows.
Assume the contrary. Let $\omega_n:X\to \omega_n(X)$ be a sequence of conformal structures of second kind on $X$ such that $X_n\stackrel{def}= \omega_n(X)$ can be identified with an increasing sequence of domains in $\mathbb{C}\setminus\{-1,1\}$ whose union equals $\mathbb{C}\setminus\{-1,1\}$. We may choose the $\omega_n$ uniformly converging on compact subsets of $X$.
We identify the fundamental groups of $X_n$ and of $\mathbb{C}\setminus\{-1,1\}$ by the isomorphism induced by inclusion.

Consider the case when the  monodromies of $f$ along $(e',e'',e''')$ are equal to $(a_1^{-1},a_2^{-1},a_2a_1)$.
By our assumption for each $n$ the mapping $f\circ \omega_n^{-1}$ is homotopic to a holomorphic mapping $f_n:X_n\to \mathbb{C}\setminus\{-1,1\}$ with monodromies $a_1^{-1}$ along $e'$ and $a_2^{-1}$ along $e''$.
By Montel's Theorem there is a subsequence $f_{n_k}$ that converges locally uniformly on $\mathbb{C}\setminus \{-1,1\}$. The limit function $F$ cannot be a constant (including the values of the constant $-1$, $1$, or $\infty$), since for curves $\gamma'$ and  $\gamma''$ representing $e'$ and $e''$, respectively, the set $f_{n_j}(\gamma')\cup f_{n_j}(\gamma'')$ separates $-1$, $1$ and $\infty$.
The limit function $F$ is a holomorphic mapping from
$\mathbb{C}\setminus \{-1,1\}$ to itself.
Hence it
extends to a meromorphic function on $\mathbb{P}^1$, and defines  therefore a branched covering of $\mathbb{P}^1$. This is impossible, since the curves $\gamma'$ and $\gamma''$ are mapped to curves that surround $-1$ and $1$, respectively, negatively.
\hfill $\Box$

\medskip

\section {Gromov-Oka mappings from tori with a hole to $ \mathfrak{P}_3$}
\label{sec:Grom.3}

A Riemann surface of genus $1$ will be called a torus
and a Riemann surface of genus $1$ with a hole will be called a torus with a hole. A smooth oriented surface of genus $1$ will be called a smooth torus.

We consider now a connected smooth oriented surface $X$ of genus one with a hole (i.e. a connected smooth oriented closed surface $X$of genus one with a point or a closed disc removed) and address the question which mappings $X\to\mathfrak{P}_n$ have the Gromov-Oka property. Recall that such a mapping can be considered as a separable quasipolynomial of degree $n$.

Recall that any continuous map from an oriented connected open smooth surface $X$ into $\mathfrak{P}_2$ has the Gromov-Oka property.

We consider Problem \ref{problGrom3} and Problem 8.4a for the case when the target manifold equals $\mathfrak{P}_3$.

We will say that a separable quasipolynomial on an open Riemann surface $X$  is isotopic to a holomorphic quasipolynomial,
if the corresponding mapping to $\mathfrak{P}_n$ is homotopic to a holomorphic one. We will say that a separable quasipolynomial $f$ on a finite open oriented smooth surface $X$  is holomorphic for the conformal structure $\omega:X\to \omega(X)$ (with $ \omega(X)$ being a Riemann surface)
if $f\circ \omega^{-1}$ is a holomorphic quasipolynomial on $\omega(X)$.
\index{conformal structure}

The soft Oka Principle states that for each separable quasipolynomial
$f$ of degree $n$ on a finite open oriented smooth surface $X$
there exists an orientation preserving  conformal
structure on $X$ for which the quasipolynomial is isotopic to a
holomorphic quasipolynomial.

In the following theorem we consider quasipolynomials of degree~3.
The theorem shows that the obstructions for a separable
quasipolynomial to be isotopic to a holomorphic quasipolynomial are
discrete.

\begin{thm}\label{thmGrom1}
Let $X$ be a torus with a hole, and let $\mathcal{E}=\{e_1,e_2\}$ be a standard system of generators of the fundamental group of $X$ with base point $q_0$. Denote by $\mathcal{E'}$ the set $\{e_1,e_2, e_2 e_1^{-1}, e_2 e_1^{-2}, e_1 e_2 e_1^{-1} e_2^{-1}\}$.

Let $f$ be a separable quasipolynomial of degree $3$ on $X$ such that for each $e\in \mathcal{E'}$ the quasipolynomial $f$ is isotopic to an algebroid function for an orientation preserving conformal structure $w_e$ on $X$ with a holomorphic annulus  representing $e$ of conformal module larger than $\frac{\pi}{2}(\log\frac{3+\sqrt{5}}{2})^{-1}$.
Suppose $X$ is a Riemann surface of second kind.
Then the quasipolynomial is isotopic to an algebroid function
on $X$.
\end{thm}
The following statement is formally slightly stronger. It follows from the proof of Theorem \ref{thmGrom1}. \\
{\it Let $X$,  $\mathcal{E}$ and  $\mathcal{E'}$ be as in the theorem.
A continuous
mapping $f:X\to \mathfrak{P}_3$ has the Gromov-Oka property if and only if for each $e\in \mathcal{E'}$ the restriction of the mapping to an annulus in X that represents $\hat e$ has the Gromov-Oka property.}
\index{$\mathcal{E'}$}

The crucial step for the proof of Theorem \ref{thmGrom1} is Theorem \ref{thmGrom2} below. The rest of this chapter is devoted to the proof of Theorem \ref{thmGrom2}. Theorem \ref{thmGrom1} and  results related to bundles will be proved in the next chapter. We first provide two lemmas that are needed for the proof of Theorem \ref{thmGrom2}.

The proof of the following lemma on permutations can be extracted for
instance from \cite{W}. For convenience of the reader we give the
short argument.
\begin{lemm}\label{lemGrom5} Let $n$ be a prime number. Any Abelian subgroup of
the symmetric group ${\mathcal S}_n$ which acts transitively on a
set consisting of $n$ points is generated by an $n$-cycle.
\end{lemm}
\noindent {\bf Proof.} Let ${\mathcal S}_n$ be the group of
permutations of elements of the set $\{ 1,\ldots , n \}$. Suppose
the elements $s_j \in {\mathcal S}_n$, $j=1,\ldots , m$, commute and
the subgroup $\langle s_1 , \ldots , s_m \rangle$ of ${\mathcal
S}_n$ generated by the $s_j$, $j = 1,\ldots , m$, acts transitively
on the set $\{ 1,\ldots , n\}$. Let $A_{s_1} \subset \{ 1,\ldots ,
n\}$ be a minimal $s_1$-invariant subset. Then $s_1 \mid A_{s_1}$ is
a cycle of length $k_1 = \vert A_{s_1} \vert$. (The order $\vert A
\vert$ of a set $A$ is the number of elements of this set.) For any
integer $\ell$ the set $s_2^{\ell} (A_{s_1})$ is minimal
$s_1$-invariant. Hence, two such sets are either disjoint or equal.
Take the minimal union $A_{s_1 s_2}$ of sets of the form $s_2^{\ell}
(A_1)$ for some integers $\ell$, which contains $A_{s_1}$ and is
invariant under $s_2$. Then $s_2$ moves the $s_2^{\ell} (A_{s_1})$
along a cycle of length $k_2$ such that $\vert A_{s_1 s_2} \vert =
k_1 \cdot k_2$. $A_{s_1 s_2}$ is a minimal subset of $\{ 1,\ldots ,
n\}$ which is invariant for both, $s_1$ and $s_2$. Continue in this
way. We obtain a minimal set $A = A_{s_1 \ldots s_m} \subset
\{1,\ldots , n\}$ which is invariant under all $s_j$. By the
transitivity condition $A = \{1,\ldots , n\}$. The order $\vert A
\vert$ equals $k_1 \cdot \ldots \cdot k_m$. Since $n$ is prime,
exactly one of the factors, $k_{j_0}$ equals $n$, the other factors
equal $1$. Then for some number $j_0$ the element
$s_{j_0}$ is a cycle of length $n$ and, if $j_0\neq 1$, then $\vert A_{s_1 \ldots s_{j_0 - 1}} \vert = 1$. Since each $s_j$ commutes with
$s_{j_0}$, each $s_j$ is a power of $s_{j_0}$. Indeed, consider a
bijection of $ \{1,\ldots , n\}$ onto the set of $n$-th roots of
unity so that $s_{j_0}$ corresponds to rotation by the angle
$\frac{2\pi}n$. In other words, put $\zeta = e^{\frac{2\pi i}n}$.
The permutation $s_{j_0}$ acts on $\{ 1 , \zeta , \zeta^2 , \ldots ,
\zeta^{n-1}\}$ by multiplication by $\zeta$. Consider an arbitrary
$s_j$. Then for some integer $\ell_j$, $s_j (\zeta) =
\zeta^{\ell_j}$, i.e. $s_j (\zeta) = \zeta^{\ell_j-1} \cdot \zeta =
(s_{j_0})^{\ell_j-1} (\zeta)$. Then for any other $n$-th root of
unity $\zeta^m$
\begin{align*}
s_j (\zeta^m) = & s_j ((s_{j_0}) \ ^{m-1}(\zeta))\\
= & (s_{j_0}) \
^{m-1}(s_j(\zeta))
=  \zeta^{m-1} \cdot \zeta^{\ell_j}
=  \zeta^{\ell_j-1} \cdot \zeta^m = (s_{j_0})^{\ell_j-1} (\zeta^m) \, .
\end{align*}
Hence $s_j = (s_{j_0})^{\ell_j-1}$. \hfill $\Box$

\smallskip

\begin{lemm}\label{lemGrom6} Let $e_1 , e_2$ be generators of a free group $F_2$.
Let $\mathcal{E'} \subset F_2$ be the finite subset $\mathcal{E'} = \{ e_1 , e_2 , e_2 \,
e_1^{-1} , e_2 \, e_1^{-2} \}$ of primitive elements of the group $F_2$. Put $\mathcal{E'}^{-1} \stackrel{def}{=}
\{e^{-1} : e \in \mathcal{E'} \}$. Suppose $\Psi : F_2 \to {\mathcal S}_3$ is a
homomorphism from $F_2$ into the symmetric group ${\mathcal S}_3$
whose image is an Abelian subgroup of ${\mathcal S}_3$ which acts
transitively on the set of three elements. Then there are elements
${\sf e}_1 , {\sf e}_2 \in \mathcal{E'} \cup \mathcal{E'}^{-1}$ such that $\Psi ({\sf e}_1)$
is a $3$-cycle, $\Psi ({\sf e}_2) = {\rm id}$ and the commutator
$[{\sf e}_1 , {\sf e}_2]$ is conjugate to $[e_1 , e_2]$.
\end{lemm}


\noindent {\bf Proof.} By
Lemma \ref{lemGrom5} the
image of one of the original generators of $F_2$ is a $3$-cycle. If $\,\Psi
(e_1)$ is a $3$-cycle, then $\Psi (e_2 \, e_1^{-q})$ is the identity
for $q$ being either $0$ or $1$ or $2$. (Recall that $\Psi (e_2)$ is
a power of $\Psi (e_1)$ and $\Psi (e_1)^3 = {\rm id}$.) Also,
$$
[e_1 , e_2 \, e_1^{-q}] \, =\,  e_1 \, e_2 \, e_1^{-q} \, e_1^{-1}
\, e_1^q \, e_2^{-1}\, = \, [e_1 , e_2] \, .
$$
We may take the pair ${\sf e}_1=e_1$, ${\sf e}_2=e_2 \, e_1^{-q}$.

If $\Psi (e_1)$ is the identity then $\Psi (e_2)$ is a $3$-cycle.
The pair $\,{\sf e}_1 = e_2\,$, ${\sf e}_2 = e_1^{-1}\,$
generates $F_2$ and the commutator $\,[{\sf e}_1 , {\sf e}_2]\, =\, e_2 \,
e_1^{-1} \, e_2^{-1} \, e_1\,$ is conjugate to the commutator
$\,[e_1 , e_2]\,.\,$ \hfill $\Box$

\medskip

Let $X$ be a torus with a disc removed. Its fundamental group $\pi_1
(X,q_0)$ with base point $q_0$ is isomorphic to the free group $F_2$
with two generators.

\begin{thm}\label{thmGrom2} Suppose $X$ is a torus with a hole.
Let $\mathcal{E}=\{e_1,e_2\}$ be a standard system of generators of the fundamental group of $X$ with base point $q_0$, and $\mathcal{E'}= \{e_1,e_2, e_2 e_1^{-1}, e_2 e_1^{-2}, e_1 e_2 e_1^{-1} e_2^{-1}\}$. Let $f$ be a separable quasipolynomial of degree $3$ on $X$ such that for each $e\in \mathcal{E'}$ the quasipolynomial $f$ is isotopic to an algebroid function for an
orientation preserving conformal structure $w_e$ on $X$ with a holomorphic annulus  representing $\hat e$ of conformal module larger than $\frac{\pi}{2}(\log\frac{3+\sqrt{5}}{2})^{-1}$.
Then the isotopy class of $f$
corresponds to the conjugacy class of a homomorphism
$$
\Phi : \pi_1 (X , q_0) \to \Gamma \subset {\mathcal B}_3
$$
for a subgroup $\Gamma$ of ${\mathcal B}_3$ \index{$\Phi$}
which is generated either by $\sigma_1 \, \sigma_2$, or by $\sigma_1 \, \sigma_2 \, \sigma_1$,
or by $\sigma_1$ and $\Delta_3^2$.

In particular, in all cases the image $\Phi ([e_1 , e_2])$ of the
commutator of the generators is the identity.
\end{thm}

For the proof of the theorem we need two lemmas which we state and prove before proving the theorem.

\begin{lemm}\label{lemGrom7} Let ${\sf b}_1 = (\sigma_1 \, \sigma_2)^{\pm 1}$ and
${\sf b}_2 = w^{-1} \, \sigma_1^{2k} \, w$ for an integer $k$ and a
braid $w \in {\mathcal B}_3$. If the commutator $[{\sf b}_1 , {\sf
b}_2]$ is conjugate to $\sigma_1^{2k^*} \Delta_3^{2\ell^*}$ for
integers $k^*$ and $\ell^*$ then ${\sf b}_2$ is the identity and
hence also the commutator is the identity.
\end{lemm}

\smallskip

\noindent {\bf Proof of Lemma \ref{lemGrom7}.} Let $b \in {\mathcal B}_n$ be a
pure braid with base point $E_n\in C_n(\mathbb{C})\diagup \mathcal{S}_n$.
Fix a point $(x_1,\ldots,x_n)\in  C_n(\mathbb{C})$ that projects to $E_n$. Label each strand of $b$ by the number $j$ of its initial point $x_j$.
Let $\{i,j\}$ be an unordered pair of distinct integer numbers between $1$ and $n$. The linking number $\ell_{\{i,j\}}$ of the
$i$-th and $j$-th strand is defined as follows. Discard all strands except the
$i$-th and the $j$-th strand. We obtain a pure braid $\sigma^{2m}$.
We call the integral number $m$ the linking number of the two
strands and denote it by $\ell_{\{i,j\}}$.

Note that the linking numbers ${\ell}_{\{i,j\}}^*$ of the braid
$\sigma_1^{2k^*} \Delta_3^{2\ell^*}$ are equal to ${\ell}_{\{12\}}^* = k^* + \ell^*$, ${\ell}_{\{23\}}^* =  \ell^*$ ,
${\ell}_{\{13\}}^* = \ell^*$. Since the braid is conjugate to a
commutator, the sum of exponents of generators in a word representing it must be zero. This means that $2k^* + 6\ell^* = 0$. Hence,
the (unordered) collection of linking numbers of $\sigma_1^{2k^*}
\Delta_3^{2\ell^*}$ is $\{ -2\ell^* , \ell^* , \ell^*\}$. Since
conjugation only permutes linking numbers between pairs of strands
of a pure braid, the unordered collection of linking numbers of
pairs of strands of $[{\sf b}_1 , {\sf b}_2]$ equals $\{ -2\ell^* ,
\ell^* , \ell^* \}$ for the integer $\ell^*$. For the pure braid
$\sigma_1^{2k}$ the linking number between the first and the second
strand equals $k$, the linking numbers of the remaining pairs of
strands equal zero.

Consider a not necessarily pure braid $w \in \mathcal{B}_3$
with base point $E_3$ and fix a lift $(x_1,x_2,x_3)$ of $E_3$.
Let $S_w$ be the permutation $S_w
= \tau_3 (w)$.
The permutation $S_w$ defines
a permutation $s_w$ acting on the set
$(1,2,3)$, so that $S_w((x_1,x_2,x_3))=(x_{s_w(1)} , x_{s_w(2)} ,
x_{s_w(3)})$.
Order the linking numbers of pairs of strands of a pure braid
$b \in \mathcal{B}_3$ as
$(\ell_{\{2,3\}},
\ell_ {\{1,3\}} , \ell_{\{1,2\}})$.
Notice that the complement of the set $\{2,3\}$ (of the set  $\{1,3\}$, or $\{1,2\}$, respectively)
in $\{1,2,3\}$ is $\{1\}$ ( $\{2\}$, or  $\{3\}$, respectively).
The ordered tuple of linking numbers of pairs of strands of $w^{-1}
\, b \, w$ equals
$$
S_w((\ell_{\{2,3\}} , \ell_{\{1,3\}} , \ell_ {\{1,2\}}))\, = \,
  (\ell_{\{s_w(2) s_w(3)\}} , \ell_{\{s_w(1) s_w(3)\}} ,
  \ell_{\{s_w(1)s_w(2)\}}).
$$
Hence, the ordered tuple of linking numbers between pairs of strands
of ${\sf b}_2 = w^{-1} \, \sigma_1^{2k} \, w$ equals $S_w (0,0,k)$.
Similarly, the ordered tuple of linking numbers between pairs of
strands of ${\sf b}_2^{-1}$ is equal to $S_w (0,0,-k)$.

Consider the commutator ${\sf b}_1 \, {\sf b}_2 \, {\sf b}_1^{-1} \,
{\sf b}_2^{-1}$.

\begin{figure}[h]
\begin{center}
\includegraphics[width=55mm]{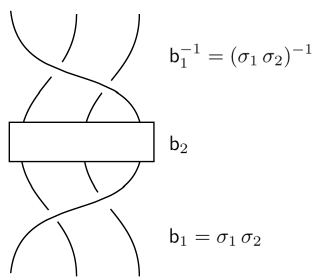}
\end{center}
\caption{The braid $b_1 b_2 b_1^{-1}= (\sigma_1 \, \sigma_2) b_2 (\sigma_1 \, \sigma_2)^{-1}$ }\label{fig8.1}
\end{figure}

The ordered tuple of linking numbers of pairs of strands of the pure
braid ${\sf b}_1 \, {\sf b}_2 \, {\sf b}_1^{-1}$ equals $(S')^{-1}
\circ S_w (0,0,k)$ for the permutation $S' = \tau_3 ({\sf b}_1)$
(see Figure \ref{fig8.1}).

\noindent Hence the ordered tuple of linking numbers of pairs of
strands of the commutator ${\sf b}_1 \, {\sf b}_2 \, {\sf b}_1^{-1}
\circ {\sf b}_2^{-1}$ equals $(S')^{-1} \circ S_w
(0,0,k) + S_w (0,0,-k)$. Since $S'$ is a 3-cycle, the mapping  $S' (x_1 , x_2 ,
x_3) = (x_{s'(1)} , x_{s'(2)} , x_{s'(3)})$ does not fix (setwise) any pair
of points among the $x_1, x_2, x_3$. Hence the unordered
$3$-tuple of linking numbers of ${\sf b}_1\, {\sf b}_2 \, {\sf
b}_1^{-1} \, {\sf b}_2^{-1}$ is $\{ k,-k,0 \}$. It can coincide with
an unordered $3$-tuple of the form $\{-2 \ell^* , \ell^* , \ell^*\}$
only if $k=\ell^*=0$. Hence $\textsf{b}_2$ is the identity and the
commutator is the identity. \hfill $\Box$

\begin{lemm}\label{lemGrom8} The centralizer of $\sigma_1^k$ in $\mathcal{B}_3$, $k \ne 0$
an integral number, equals $\{\sigma_1^{k'} \, \Delta_3^{2\ell'} :
k',\ell' \in {\mathbb Z}\}$.
\end{lemm}

\noindent {\bf Proof.} Consider for each $b\in {\mathcal B}_3$ the modular transformation $T_b$ on the Teichm\"uller space ${\mathcal T}(0,4)\cong \mathbb{C}_+$ that is associated to a mapping $\varphi_ {b,\infty}\in \mathfrak{m}_ {b,\infty}$. Recall that the mapping ${\mathcal B}_3\diagup \mathcal{S}_3\ni b\diagup \mathcal{S}_3\to T_b$ is a bijection that satisfies equality \eqref{eq2.100} (see Section \ref{sec:2.3a}). The group of modular transformations on $ \mathbb{C}_+$ is isomorphic to $SL_2(\mathbb{Z})\diagup \pm {\rm Id}$. Recall that $T_{\sigma_1}(z)=\frac{z}{1-z}$ corresponds to
$ \begin{pmatrix} 1&0 \\
-1&1 \end{pmatrix}\diagup \pm {\rm Id}$ (see Section \ref{sec:4.1b}, Example 5). Suppose the braid $b$ commutes with $\sigma_1^k$ for some non-zero integer $k$ and corresponds to
$V\diagup \pm {\rm I}$ where    $ {\sf V} =  \begin{pmatrix} v_{11}&v_{12} \\
v_{21}&v_{22} \end{pmatrix}.$
Then
$$
\begin{pmatrix} 1&0 \\ -k&1 \end{pmatrix} \begin{pmatrix} v_{11}&v_{12} \\ v_{21}&v_{22} \end{pmatrix} = \begin{pmatrix} v_{11}&v_{12} \\ v_{21}&v_{22} \end{pmatrix} \begin{pmatrix} 1&0 \\ -k&1 \end{pmatrix}
$$
i.e.
$$
\begin{pmatrix} v_{11} &\;\; v_{12}  \\-k\, v_{11}+ v_{21}&\; \;-k \,v_{12}+v_{22} \end{pmatrix} =  \begin{pmatrix} v_{11}-k\, v_{12} &\;\;  v_{12} \\ v_{21}-k\, v_{22}&\;\; v_{22} \end{pmatrix} \, .
$$
Since $k \ne 0$ we have $v_{12} = 0$ and $v_{11} = v_{22}$. Since
$\det {\sf V} = 1$ we obtain
$
{\sf V} = \pm \begin{pmatrix} 1&0 \\ -m&1 \end{pmatrix} $ for an integer $m$. Thus $b= \sigma_1^m $
for some integer $m$.
The
lemma is proved. \hfill $\Box$

\bigskip

\noindent {\bf Proof of Theorem \ref{thmGrom2}.} Let $\Phi : \pi_1 (X,x_0) \to {\mathcal B}_3$
be a homomorphism whose conjugacy class corresponds to the isotopy
class of $f$. Denote by $\Psi = \tau_3 \circ \Phi : \pi_1 (X,x_0)
\to {\mathcal S}_3$ the related homomorphism into the symmetric
group. Since the quasipolynomial $F$ is isotopic to an algebroid function for the conformal structure $\omega_{[e_1,e_2]}$ with an annulus of conformal module larger than $\frac{\pi}{2}(\log\frac{3+\sqrt{5}}{2})^{-1}$ representing the commutator $[e_1,e_2]$,
by Lemma \ref{lemm8.3} the conformal module of the commutator $[b_1,b_2]$ equals infinity. Hence, by Lemma
\ref{lemm8.3}
$\,\Psi ([e_1 , e_2]) \,= \,{\rm id}\,.\,$
\index{$\Psi$}

Consider first the case when the subgroup $\Psi
(\pi_1 (X,x_0))$ of ${\mathcal S}_3$ acts transitively on the set of
three points. Apply Lemma \ref{lemGrom6} to the free group $\pi_1
(X,x_0)$ with generators $e_1$ and $e_2$ and to the homomorphism
$\Psi$. We obtain new generators ${\sf e}_1 , {\sf e}_2 \in \mathcal{E'} \cup
\mathcal{E'}^{-1}$ of $\pi_1(X,x_0)$ (with $\mathcal{E'}$ and $\mathcal{E'}^{-1}$ being the sets of the lemma) such that
$\Psi ({\sf e}_1)$ is a $3$-cycle and $\Psi ({\sf e}_2) = {\rm id}$.
Put $s_1 = \Psi ({\sf e}_1)$, $s_2 = \Psi ({\sf e}_2)$, $b_1 = \Phi
({\sf e}_1)$, $b_2 = \Phi ({\sf e}_2)$.
Then $b_2$ and $[b_1,b_2]$
are pure braids (the latter holds since $[{\sf e}_1 , {\sf e}_2]$ is
conjugate to $[e_1 , e_2]$). Since $f$ is isotopic to an algebroid
function for the conformal structure $w_{{\sf e}_1}$,  Lemma \ref{lemm8.3} implies  ${\mathcal M} (b_1) = \infty$. Since $\tau_3 (b_1)$ is a $3$-cycle, by the same lemma
the braid $b_1$ must be conjugate to an integral power of $\sigma_1
\, \sigma_2$.

After conjugating $\Phi$ we may assume that $b_1 = (\sigma_1 \,
\sigma_2)^{3\ell \pm 1} = (\sigma_1 \, \sigma_2)^{\pm 1} \cdot
\Delta_3^{2\ell}$ for an integer $\ell$. By Lemma \ref{lemm8.3} we also have
${\mathcal M} (b_2) = \infty$ for the pure braid $b_2$. Hence, $b_2
= w^{-1} \, \sigma_1^{2k} \, \Delta_3^{2\ell'} \, w$ for integers
$k$ and $\ell'$ and a conjugating braid $w \in {\mathcal B}_3$.
Since ${\mathcal M} ([b_1 , b_2]) = \infty$ the commutator $[b_1 , b_2]$ is
conjugate to $\sigma_1^{2k^*} \Delta_3^{2\ell^*}$ for integers $k^*$
and $\ell^*$.

Take into account that $\Delta_3^{2\ell}$ commutes with each $3$-braid
and apply Lemma \ref{lemGrom7} to ${\sf b}_1= (\sigma_1 \, \sigma_2)^{\pm 1}$ and ${\sf b}_2= w^{-1} \sigma_1^{2k^*} w$.
This gives the statement of Theorem \ref{thmGrom2} for the case
when the subgroup $\Psi(\pi_1(X,x_0))$ acts transitively on the set of three points.

\medskip

Consider now the case when the subgroup $\Psi
(\pi_1 (X,x_0))$ of ${\mathcal S}_3$ does not act transitively on the set of
three points.
Then $\Psi (\pi_1 (X,x_0))$ is generated either by a
transposition (we may assume the transposition to be $(12)$ by
choosing the mapping  $\Phi$ in its conjugacy class) or is equal to
the identity in ${\mathcal S}_3$.

We may suppose that $\Psi (\pi_1 (X,x_0))$ is not the identity. There are generators ${\sf e}_1 , {\sf e}_2 \in \mathcal{E'} \cup \mathcal{E'}^{-1}$ such
that $[{\sf e}_1 , {\sf e}_2]$ is conjugate to $[e_1 , e_2]$ and
$\Psi ({\sf e}_2) = {\rm id}$. Indeed, suppose $\Psi(\pi_1(X;x_0))$
is generated by the transposition $(12)$ and neither $\Psi(e_1)$ nor
$\Psi(e_2)$ is the identity. Then $\Psi(e_1)=\Psi(e_2)=(12).\,$ Hence,
$\Psi(e_2 e_1^{-1})=\rm {id}$,
and for the generators ${\sf e}_1\stackrel{def}=e_1 $ and $ {\sf e}_2\stackrel{def}=e_2 e_1^{-1}$ the claimed statements hold.

Use the notation $b_1 = \Phi ({\sf e}_1)$, $b_2 = \Phi ({\sf e}_2)$. By Lemma
\ref{lemm8.3} we may assume, that $b_1$ is conjugate to either
$\sigma_1^{ k_1} \, \Delta_3^{2\ell_1}$ or to $(\sigma_1 \sigma_2
\sigma_1)^{k_1}= \Delta_3^{k_1}\,$ for an odd integer $k_1$ and an integer $\ell_1$,
and $b_2$ is conjugate to $\sigma_1^{2 k_2} \,
\Delta_3^{2\ell_2}$ for integers $k_2$ and $\ell_2$.
Conjugating $\Phi$, we may assume that
$b_1 = B_1 \, \Delta_3^{2 \ell_1}$ with either $B_1 =\sigma_1 ^{k_1}$ or
$B_1 = \sigma_2 \, \sigma_1^2 \,=\, \sigma_1^{-1} (\sigma_1 \sigma_2
\sigma_1) \sigma_1$.
Respectively, $b_2$ is conjugate to $B_2\,\Delta_3^{2 \ell_2}$ with $B_2= \sigma_1^{2 k_2}$. 
Since $[b_1 , b_2]$ is a pure braid with infinite conformal module,
Lemma \ref{lemm8.3} implies that $[b_1 , b_2]$ is conjugate to
$\sigma_1^{2k^*} \, \Delta_3^{2\ell^*}$. Since the commutator has
degree zero the equality $2k^* + 6\ell^* = 0$ holds and the
unordered tuple of linking numbers of $[b_1 , b_2]$ equals $\{ \ell^* , \ell^*
,-2\ell^* \}$.

This implies that the commutator $[b_1 , b_2]$ 
and the braid $B_2$ are equal to the identity.
Indeed, suppose $b_2=w^{-1} \, \sigma_1^{2 k_2} \, w \,
\Delta_3^{2 \ell_1}\,$ for a braid $w \in \mathcal{B}_3.\,$ The
ordered tuple of linking numbers of pairs of strands of $B_2=w^{-1} \, \sigma_1^{2 k_2} \, w$ equals
$S_w (0,0,k_2)$ and the respective ordered triple for $B_2^{-1}$ is
$S_w (0,0,-k_2)$. Here $S_w = \tau_3(w)$.
The ordered tuple of
linking numbers of pairs of strands of $B_1 \, B_2 \, B_1^{-1}$
equals $ (S')^{-1} \circ S_w (0,0,k_2)$, where $S' =
\tau_3 (B_1)$. 
Since $\Delta_3^2$ is  in the center of
$\mathcal{B}_3$,
the unordered tuple of linking numbers of pairs of strands of
$[b_1 , b_2]$ is equal to that of $[B_1 , B_2]$, i.e. it equals either $\{ k_2 , -k_2 , 0\}$ or $\{0,0,0\}$
(in dependence on $S_w$). Either of these unordered tuples can be
equal to $\{\ell^* , \ell^* , -2\ell^*\}$ for some integer $\ell^*$
only if $\ell^* = 0$. We obtained that $[b_1,b_2] = \rm{id}$ also in
this case. By Lemma \ref{lemGrom8} the group $\Phi(\pi_1(X,x_0))$ is generated
either by (a power of) $\sigma_1$ and $\Delta_3^2$, or by (a power of) $\sigma_2 \,
\sigma_1^2\,$ and $\Delta_3^2$. The theorem is proved. \hfill $\Box$

\medskip

We do not know the answer to the following problem.
\begin{probl}\label{problGrom6}  Can two braids in ${\mathcal B}_3$ of zero entropy
have a non-trivial commutator of zero entropy?
\end{probl}
If the answer is negative, then Theorem \ref{thmGrom2} holds with $\mathcal{E'}=\mathcal{E}$.
The proof of Theorem \ref{thmGrom2}
gives the following corollary which is weaker than a negative answer
to Problem \ref{problGrom6}.
\begin{cor}\label{corGrom1}
Let ${b}_1$ and ${b}_2$ be braids in ${\mathcal{B}}_3$ with
$h(b_1)=h(b_2)=h([b_1,b_2])=0$. If one of the braids is pure or
$$h(b_2\, b_1^{-1})=h(b_2\,b_1^{-2})=0$$
then $[b_1,b_2]= \rm{id}$.
\end{cor}

On the other hand D.Calegari and A.Walker suggested the following example of two elements
of the braid group $\mathcal{B}_3$ with non-trivial commutator of
zero entropy.
\begin{ex}\label{exGrom1}
The commutator of the non-commuting braids $b_1= (\sigma_2)^{-1} \, \sigma_1\;$ and $ b_2 =
\sigma_2 \, (\sigma_1)^{-1}\,$  has entropy zero.
\end{ex}
It is enough to show that the commutator of the two braids equals $[b_1,b_2] =
(\sigma_2)^{-6} \, \Delta_3^2\,.\,$  Since
\begin{align*}
b_1^{-1} \, b_2^{-1} \,=\,&(\sigma_1)^{-1} \, \sigma_2 \, \sigma_1 \;
(\sigma_2)^{-1}\;\;\;\;\;  =\,(\sigma_1)^{-1} \, \sigma_2 \, \sigma_1
\,\sigma_2 \; (\sigma_2)^{-2} \, \\
=\,&(\sigma_1)^{-1} \, \sigma_1 \,
\sigma_2 \,\sigma_1 \; (\sigma_2)^{-2}
\, = \, \sigma_2 \, \sigma_1 \,(\sigma_2)^{-2}\, ,\\
b_1 \, b_2 \; \, =\,&(\sigma_2)^{-1} \, \sigma_1 \, \sigma_2 \;
(\sigma_1)^{-1}\;\; \;\;\; =\, (\sigma_2)^{-2} \, \sigma_2 \, \sigma_1 \,
\sigma_2 \;(\sigma_1)^{-1}\,\\
 =\,&(\sigma_2)^{-2} \, \sigma_1 \,
\sigma_2 \, \sigma_1 (\sigma_1)^{-1}
\,\; = \, (\sigma_2)^{-2} \, \sigma_1 \,\sigma_2\, ,
\end{align*}
we obtain
\begin{align*}
[b_1,b_2] = \, b_1 b_2 b_1^{-1}b_2^{-1} \, = \,&(\sigma_2)^{-3} \,
\sigma_2 \, \sigma_1 \, \sigma_2 \, \cdot \sigma_2 \, \sigma_1 \,
\sigma_2 \; (\sigma_2)^{-3} \,\\
=  \,&(\sigma_2)^{-3} \; \Delta_3^2 \;
(\sigma_2)^{-3}
= \, (\sigma_2)^{-6} \;\Delta_3^2\,.
\end{align*}
It is clear now that the commutator $[b_1,b_2]$ has entropy zero.


Problem \ref{problGrom6} has a positive answer for braid groups on more than $3$ strands.
Consider the braids $ b_1 =\sigma_1^{-2}$ and $b_2 = (\sigma_2)^{-1} \,
(\sigma_1)^{-1} \, (\sigma_3)^{-1} \, (\sigma_2)^{-1}$ in $\mathcal{B}_4$ of zero entropy.
Their commutator equals $[b_1
, b_2] = \sigma_1^{-2} \cdot \sigma_3^{2} \ne {\rm id}$ and is a non-trivial braid of zero entropy (see Figure \ref{Fig82}).
\begin{figure}[h]
\begin{center}
\includegraphics[width=10cm]{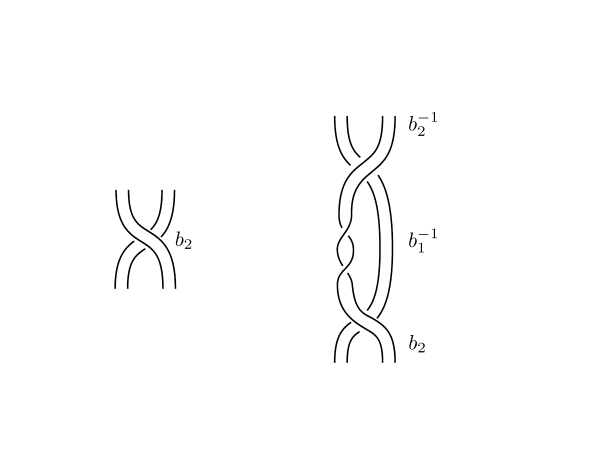}
\caption{Two $4$-braids of vanishing entropy with non-trivial commutator}\label{Fig82}
\end{center}
\end{figure}

Notice the following fact. Let $X$ be a torus with a disc removed.
Consider the separable quasipolynomial of degree four on $X$ whose
isotopy class corresponds to the conjugacy class of the homomorphism
$\Phi : \pi_1 (X,x_0) \to {\mathcal B}_4$ with $\Phi (e_1) = b_1$ and $\Phi (e_2) = b_2$.
By a similar argument as in the proof of Theorem \ref{thmGrom1}
the quasipolynomial is isotopic to an algebroid function for each
conformal structure of second kind on $X$.

\begin{probl}\label{problGrom7.} What is the analogue of Corollary
$\ref{corGrom1}$ for braid groups on more than three strands? What is the analog of Theorem $\ref{thmGrom1}$ for quasipolynomials of degree $4$?
\end{probl}

\chapter{Gromov's Oka Principle for $({\sf g},{\sf m})$-fiber bundles.}
\label{chapterEl}

In this chapter we prove theorems  related to the failure and the restricted validity of the Gromov-Oka Principle for bundles whose
fibers are Riemann surfaces of type $(1,1)$ (in other words, the fibers are once punctured tori). The problem can be reduced to the respective problem for
$(0,4)$-bundles (and, hence, for mappings into  $\mathfrak{P}_3$)  by considering double branched coverings.

\section{$({\sf g},{\sf m})$-fiber bundles. Statement of the Problem.}
\label{sec:9.1}

We will consider bundles whose fibers are connected
closed oriented surfaces of genus $\textsf{g}\geq 0$ with $\textsf{m}\geq 0$ distinguished points.
Recall that a connected smooth closed oriented surface of genus $\sf{g}$ with $\sf{m}$ distinguished
points is called a surface of type $(\sf{g},\sf{m})$. In case the surface is equipped with a complex
structure we call it a Riemann surface of type $(\sf{g},\sf{m})$.

\begin{defn}\label{def9.1}{\rm (Smooth oriented $(\textsf{g},\textsf{m})$ fiber bundles.)}
Let $X$ be a smooth oriented manifold of dimension $k$, let
${\mathcal X}$ be a smooth (oriented)
manifold of dimension $k+2$ and ${\mathcal P} : {\mathcal X} \to X$
an orientation preserving smooth proper submersion such that for each point $x \in X$ the
fiber ${\mathcal P}^{-1} (x)$ is a smooth closed oriented surface of genus $\sf{g}$.
Let $\mathbold{E}$ be a smooth submanifold of $\mathcal{X}$ that intersects each fiber $\mathcal{P}^{-1}(x)$ along a set $E_x$ of $\sf{m}$ distinguished points.
Then the tuple ${\mathfrak F}_{\textsf{g},\textsf{m}} = ({\mathcal
X} , {\mathcal P} , \mathbold{E}, X)$ is called
a smooth (oriented) fiber bundle over $X$ with fibers being smooth closed oriented surfaces of
genus $\textsf{g}$ with $\textsf{m}$ distinguished points
(for short, a smooth oriented $(\textsf{g},\textsf{m})$-bundle).
\end{defn}
If ${\sf m}=0$ the set $\mathbold E$ is the empty set and we will often denote the bundle
by $({\mathcal X} , {\mathcal P} , X)$.
If ${\sf{m}}>0$ the mapping $x\to E_x$ locally defines $\sf m$ smooth sections of the $({\sf g},0)$-bundle $({\mathcal X} , {\mathcal P} , X)$ that is obtained by forgetting all distinguished points. A section of the bundle $({\mathcal X} , {\mathcal P} , X)$ is a continuous mapping $s:X\to {\mathcal X}$ with $\mathcal{P}\circ s$ being the identity on $X$.
$(\textsf{g},0)$-bundles will also be called genus $\textsf{g}$ fiber bundles.
For $\textsf{g}=1$ and $\textsf{m}=0$ the bundle is also called an elliptic fiber
bundle.  \index{fiber bundle !  elliptic}
We will consider mostly the case when $2{\sf g}-2 + {\sf m}>0$.
\index{$\mathfrak{F}_{\textsf{g},\textsf{m}}$} \index{$\mathcal{P}$} \index{$\mathbold{E}$}
\index{$\mathcal{X}$} 

\index{$\mathfrak{F}_{{\sf g},{\sf m}}$} \index{fiber bundle ! $({\sf g},{\sf m})$-fiber bundle} \index{${\rm pr}_1$}
\index{Ehresmann ! Fibration Theorem}
Let $S$ be a smooth reference surface of genus $\sf g$ and $E\subset S$ a set of $\sf{m}$ distinguished points.
For an open subset $U$ of $X$ we consider the trivial bundle (also called product bundle) $\big(U\times S, {\rm {pr}}_1, U\times E,U\big)$
with set $\{x\}\times E$ of distinguished points in the fiber $\{x\}\times S$ over $x$. Here ${\rm pr}_1: U\times S\to U $ is the projection onto the first factor.
By Ehresmann's Fibration Theorem each smooth $(\sf{g},\sf{m})$-bundle ${\mathfrak
F}_{\textsf{g},\textsf{m}} = ({\mathcal
X} , {\mathcal P} , \mathbold{E}, X)$ with set of distinguished points $E_x \stackrel{def}=\mathbold{E}\cap{\mathcal P}^{-1}(x) $ in the fiber over $x$ is
locally smoothly trivial, i.e. each point in $X$ has a neighbourhood $U\subset X$
such that for a
surjective diffeomorphism $\varphi_U:\mathcal{P}^{-1}(U) \to U \times S$
the diagram

\[
\xymatrix{
{\mathcal{P}^{-1}(U)}      \ar[r]^{\varphi_U     } \ar[d]_{\mathcal{P}         } & {U \times
S}   \ar[dl]^{{\rm  pr}_1} \\
  U
}
\]
\index{fiber bundle ! smoothly trivial}
\noindent is commutative and $\varphi_U$ maps $\mathbold{E}\cap \mathcal{P}^{-1}(U)$ onto $U\times E$, equivalently,
for all $x\in U$ the diffeomorphism
$\varphi_U$ maps
the set of distinguished points $E_x=\mathbold{E}\cap \mathcal{P}^{-1}(x)$ in the fiber $\mathcal{P}^{-1}(x)$ to the set of
distinguished points $\{x\}\times E$ in the fiber $\{x\}\times S$.

The idea of the proof of Ehresmann's Theorem is the following. Choose smooth coordinates on $U$
by a mapping from a rectangular box in $\mathbb{R}^n$ to $U$.
Consider
smooth vector fields $v_j$ on $U$, which form a basis of the tangent
space of $U$ at each point of $U$. Take smooth vector fields $V_j$ on ${\mathcal
P}^{-1} (U)$ that are tangent to $\mathbold{E}$ at points of this set and are mapped to
$v_j$ by the differential of
${\mathcal P}$. Such vector fields can easily be obtained locally.
To obtain the globally defined vector fields $V_j$ on ${\mathcal P}^{-1} (U)$
one uses partitions of unity. The required diffeomorphism $\varphi_U$ is
obtained by composing the flows of the vector fields $V_j$ (in any fixed
order).

In this way a trivialization of the bundle can be obtained over any simply connected smooth manifold.

\smallskip

In the case when the base manifold is a Riemann surface,  a
holomorphic $(\textsf{g}$,$\textsf{m})$ fiber bundle over $X$
is defined as follows.
\begin{defn}\label{defEl2}
Let $X$ be a Riemann surface, let ${\mathcal X}$ be a complex surface, and ${\mathcal P}$
a holomorphic proper submersion from ${\mathcal X}$ onto $X$, such that
each fiber $\mathcal{P}^{-1}(x)$ is a closed Riemann surface of genus $\sf g$.
Suppose $\mathbold{E}$ is a complex one-dimensional submanifold of $\mathcal{X}$ that intersects each fiber $\mathcal{P}^{-1}(x)$ along a
set $E_x$ of $\textsf{m}$ distinguished points. Then the
tuple ${\mathfrak F}_{\textsf{g},\textsf{m}} = ({\mathcal X} , {\mathcal P}, \mathbold{E}, X)$ is called a
holomorphic $(\textsf{g}$,$\textsf{m})$ fiber bundle over $X$.
\end{defn}
Notice that the mapping $x\to E_x$
locally defines $\sf{m}$ holomorphic sections of the $({\sf g},0)$-bundle $({\mathcal X} , {\mathcal P}, X)$ that is obtained by forgetting all distinguished points.

We will call two smooth oriented (holomorphic, respectively) $(\textsf{g},\,\textsf{m})$ fiber
bundles, ${\mathfrak F}^0 = ({\mathcal X}^0 ,
{\mathcal P}^0 , \mathbold{E}^0, X^0)$ and ${\mathfrak F}^1 = ({\mathcal X}^1 , {\mathcal
P}^1 ,\mathbold{E}^1, X^1)$, smoothly isomorphic (holomorphically isomorphic, respectively) if there are smooth (holomorphic,
respectively) homeomorphisms $\Phi: \mathcal{X}^0\to \mathcal{X}^1$ and $\phi:X^0\to X^1$
such that for each $x \in X^0$ the mapping $\Phi$ takes the fiber $(\mathcal{P}^0 )^{-1}(x)$
onto the fiber $(\mathcal{P}^1 )^{-1}(\phi(x))$ and the set of distinguished points in
$(\mathcal{P}^0 )^{-1}(x)$ to the set of distinguished points in $(\mathcal{P}^1 )^{-1}(\phi(x))$.
Holomorphic bundles that are holomorphically isomorphic will be considered the same holomorphic bundles.
\index{fiber bundles ! holomorphically isomorphic} \index{fiber bundles ! smoothly isomorphic}

For a smooth $({\sf g},{\sf{m}})$-bundle ${\mathfrak F} = ({\mathcal X} ,
{\mathcal P} ,\mathbold{E}, X)$ over $X$ and a homeomorphism $\omega:X\to \omega(X)$ onto a surface $\omega(X)$ we denote by ${\mathfrak F}_{\omega}$ the bundle
\begin{equation}\label{eqEl23}
{\mathfrak F}_{\omega} = ({\mathcal X} ,
\omega\circ{\mathcal P} ,\mathbold{E},\omega(X) ).
\end{equation}
The bundles ${\mathfrak F}$ and ${\mathfrak F}_{\omega}$ are isomorphic. Indeed, we may consider the homeomorphisms
 $\omega:X\to \omega(X)$ and ${\rm Id}:{\mathcal X}\to {\mathcal X}$.

We are interested in smooth deformations of a smooth bundle to a holomorphic one.
Two smooth (oriented) $(\textsf{g},\textsf{m})$ fiber bundles
over the same oriented
smooth base manifold $X$,  ${\mathfrak F}^0 = ({\mathcal X}^0 ,
{\mathcal P}^0 ,\mathbold{E}^0, X)$, and ${\mathfrak F}^1 = ({\mathcal X}^1 , {\mathcal
P}^1 ,\mathbold{E}^1, X)$, are called (free) isotopic if for an open interval $I$ containing $[0,1]$ there
is a smooth $(\textsf{g},\textsf{m})$ fiber bundle $({\mathcal Y} , {\mathcal P} , \mathbold{E},  I \times
X)$ over the base $I \times X$ (called an isotopy) with the following property.  For each $t \in [0,1]$ we put
${\mathcal Y}^t = {\mathcal P}^{-1} (\{t\}\times X)$ and $\mathbold{E}^t= \mathbold{E}\cap {\mathcal P}^{-1} ( \{t\}\times X)$. The bundle ${\mathfrak F}^0\,$ is
equal to
$\,\left( {\mathcal Y}^0 , {\mathcal P} \mid {\mathcal Y}^0 , \mathbold{E}^0,   \{0\}
\times  X\right)\,$, and the bundle $\,{\mathfrak F}^1\,$ is
equal to $\,\bigl( {\mathcal Y}^1 , {\mathcal P} \mid {\mathcal
Y}^1 , \mathbold{E}^1,  \{1\} \times X \bigl)\,$.

\index{isotopy ! of fiber bundles}
\index{fiber bundles ! isotopic}
Notice that for
each $t \in I$ the tuple $\left( {\mathcal Y}^t , {\mathcal P} \mid {\mathcal Y}^t , \mathbold{E}^t,  \{t\}
\times X\right)$ is automatically a smooth  $({\sf{g}},{\sf{m}})$-fiber bundle.

\begin{lemm}\label{lemEl2}
Isotopic bundles are isomorphic.
\end{lemm}
{\bf Proof.}
Let $I \supset [0,1]$ be an open interval, and let   $(\mathcal{Y},\mathcal{P},\mathbold{E},I\times X)$ be a smooth $\sf(g,m)$-bundle over $I\times X$ such that the restrictions to $\{0\}\times X$ and to $\{1\}\times X$ are equal to given smooth $\sf(g,m)$-bundles over $X$. Here $X$ is a smooth finite open surface.
Consider the vector field $v$ on $I\times X$ that equals the unit vector in positive direction of $I$ at each point of $I\times X$, and let $V$ be a smooth vector field on $\mathcal{Y}$ that projects to $v$ under $\mathcal{P}$ and is tangent to ${\mathbold{E}}$  at points of this set. For each $t\in I$ we let $\varphi_t$
be the time $t$ map of the flow of $V$, more precisely,
\begin{align*}
\frac{\partial}{\partial t}  \varphi_t(y)=V(\varphi_t(y)),\;\; \varphi_0(y)=y, \;\;\; t\in I,\; y \in \mathcal{P}^{-1}(\{0\}\times X)\,.
\end{align*}
Then $ \varphi_1$ is a diffeomorphism from $\mathcal{P}^{-1}(\{0\}\times X)$ onto $\mathcal{P}^{-1}(\{1\}\times X)$, that maps the fiber over $(0,x)$ onto the fiber over $(1,x)$ and maps distinguished points to distinguished points.
\hfill $\Box$

In Section \ref{sec:9.2} we will prove that vice versa, isomorphic bundles are isotopic.

The following problem on isotopies of smooth
objects to the respective holomorphic objects concerns another version of the restricted validity of Gromov's Oka Principle.

\begin{probl}\label{problEl.1a}
Let $X$ be a finite open Riemann surface. Can a given smooth $(\textsf{g,m)}$-fiber
bundle over $X$ be smoothly deformed to a
holomorphic fiber bundle over $X$?

More precisely,, does there exist an isotopy $({\mathcal Y} , {\mathcal P} , \mathbold{E},  X \times
I)$ over the base $X \times I$ with $[0,1]\subset I$, for which the restriction to $\{0\}\times X$ is equal to the given bundle, and the restriction to $\{1\}\times X$ can be equipped with the structure of a holomorphic bundle?
\end{probl}
The answer is in general negative. There is a similar notion that describes obstructions for the existence of such
isotopies for $(\textsf{g,m)}$-bundles as for the existence of isotopies of smooth separable
quasipolynomials to holomorphic ones.
It is called the conformal module of isotopy classes of
$(\textsf{g,m)}$-fiber bundles over the circle, and is defined as follows.
\index{conformal module ! of isotopy
classes of fiber bundles} Consider a smooth oriented $({\sf g},{\sf m})$-fiber
bundle ${\mathfrak F} = {\mathfrak F}_{({\sf g},{\sf m})} = ({\mathcal X} , {\mathcal
P} , \mathbold{E},
\partial {\mathbb D})$ over the circle. Denote by
$\widehat{\mathfrak F}$ the \index{$\widehat{\mathfrak F}$}
isotopy class of fiber bundles over $\partial {\mathbb D}$ that
contains ${\mathfrak F}$. Let $A_{r,R} = \{ z \in {\mathbb C} : r <
\vert z \vert < R \}$, $r < 1 < R$, be an annulus  containing the
unit circle. A smooth $({\sf g},{\sf m})$-fiber bundle on $A_{r,R}$ is said to represent
$\widehat{\mathfrak F}$ if its restriction to the unit circle
$\partial {\mathbb D}$ is an element of $\widehat{\mathfrak F}$.

\begin{defn}\label{defEl3}
{\rm (The conformal module of isotopy classes of $({\sf g},{\sf m})$-fiber
bundles.)} Let $\widehat{\mathfrak F} = \widehat{\mathfrak F}_{{\sf g},{\sf m}}$ be
the isotopy class of an oriented $({\sf g},{\sf m})$-fiber bundle over
the circle $\partial {\mathbb D}$. Its conformal module is defined
as
\begin{align}\label{eqEl8}
{\mathcal M} (\widehat{\mathfrak F}) = \sup \{ m(A_{r,R}) :\, & \mbox{there exists a
holomorphic fiber bundle} \nonumber\\ & \mbox{on}\; A_{r,R}\; \mbox{that represents}\;
\widehat{\mathfrak F}\}.
\end{align}
\end{defn}

\section[Monodromies. Isotopy and isomorphism of bundles]  {The monodromy of $(\sf{g},\sf{m})$-bundles.  Isotopy and isomorphism of bundles.}\label{sec:9.2}
\noindent {\bf Monodromy and maping torus.} Consider a smooth $(\sf{g},\sf{m})$-bundle
$\mathfrak{F}_{\sf{g},\sf{m}}=(\mathcal{X},\mathcal{P},\mathbold{E},\partial{\mathbb{D}})$ over the unit
circle $\partial{\mathbb{D}}$. Denote as before the set of distinguished points $\mathbold{E}\cap \mathcal{P}^{-1}(x)$ in the fiber
over $x$ by $E_x$.
Let $v$ be the unit tangent vector field to $\partial
{\mathbb D}$. The argument used for the proof of Ehresmann's
Theorem provides a smooth vector field
$V$ on ${\mathcal X}$ which is  tangent to $\mathbold{E}$ at points of this set
and projects to $v$, i.e. $(d \, {\mathcal P})(V) = v$. Cover $\partial {\mathbb
D}$ by its universal covering $\mathbb{R}\overset{p}\to\partial{\mathbb D}$, using the mapping $p(t)=
 e^{2\pi i t}$, $t \in {\mathbb R}$. Lift the bundle ${\mathfrak
F}_{\sf{g},\sf{m}}$ over $\partial {\mathbb D}$ to a bundle $\widetilde{\mathfrak
F}_{\sf{g},\sf{m}} = (\widetilde{\mathcal X} , \widetilde{\mathcal P} , \widetilde{\mathbold{E}}, {\mathbb
R})$ over ${\mathbb R}$.

The fiber $\widetilde{\mathcal{P}}^{-1}(t)$ with set of distinguished points $\tilde{E}_t$ equals
$\mathcal{P}^{-1}(e^{2\pi i t})$ with set of distinguished points $E_{e^{2\pi i t}}$, hence for $k
\in \mathbb{Z}$ and each $t \in \mathbb{R}$ the sets
$\widetilde{\mathcal{P}}^{-1}(t+2\pi k)$ and   $\widetilde{\mathcal{P}}^{-1}(t)$,
are equal, and $\tilde{E}_{t+2\pi k}$ is equal to $\tilde{E}_{t}$.
Lift the vector field $V$ to a
vector field $\widetilde V$ on $\widetilde{\mathcal{X}}$ with the following property. The equality
$\widetilde V(z_1) =\widetilde V(z_2) $ holds, if $z_1, z_2 \in  \widetilde{\mathcal{X}}$ are
mapped to the same point $z  \in \mathcal{X}$ under the projection $\widetilde{\mathcal{X}}\overset{{\tilde p}}{-\!\!\!\longrightarrow}
{\mathcal{X}}$.

For $t \in \mathbb{R},\, \zeta \in \widetilde{\mathcal{X}}$,
we let $\widetilde\varphi_t (\zeta)\in \widetilde{\mathcal{X}}$  be
the solution of the differential equation
\begin{equation}\label{eqEl1}
\frac\partial{\partial t} \, \widetilde\varphi_t(\zeta) = \widetilde
V (\widetilde\varphi_t (\zeta)),\; \widetilde\varphi_0(\zeta) = \zeta  .
\end{equation}
Put $S \stackrel{def}{=}{\mathcal P}^{-1} (1)\cong \widetilde{\mathcal P}^{-1} (0)$ and $E\stackrel{def}= \mathbold{E}\cap  {\mathcal P}^{-1} (1) \subset S$. Let
$\zeta \in S$. Then $\widetilde \varphi_t(\zeta) \in  \widetilde{\mathcal P}^{-1} (t)$. The
time $t$ map $\widetilde\varphi_t (\zeta)$, $\zeta \in S ,$
of the vector field $\widetilde V$ defines a homeomorphism from the
fiber $S \cong \widetilde{\mathcal P}^{-1} (0)$ onto the fiber
$\widetilde{\mathcal P}^{-1} (t)$, that maps the set of distinguished points $E\subset S$ to the set of distinguished points $\widetilde{E }_t\subset \widetilde{\mathcal P}^{-1} (t)$. In the same way $\widetilde \varphi_t$ defines a
homeomorphism from the fiber $\widetilde{\mathcal P}^{-1} (t_1)$ onto the fiber
$\widetilde{\mathcal P}^{-1} (t_1+t)$ that maps the distinguished points $\tilde{E }_{t_1}$ to the distinguished points $\tilde{E }_{t_1+t}$. The mappings $\widetilde \varphi_t$ form a group:
\begin{equation}\label{eqEl2}
\widetilde \varphi_{t_1+t_2}(\zeta) = \widetilde \varphi_{t_1}(\widetilde \varphi_{t_2}(\zeta)),
\;  \zeta \in \widetilde{\mathcal{X}}, \; t_1, t_2 \in \mathbb{R}.
\end{equation}
Let $\big(\mathbb{R} \times S ,\, {\rm pr}_1\,, \mathbb{R}\times E,\, \mathbb{R}\big)$ be the trivial bundle.
Here ${\rm pr}_1:\mathbb{R} \times S \to \mathbb{R}$ is the projection onto the first factor.
The mapping $\Phi$,
\begin{equation}\label{eqEl3}
\mathbb{R}\times S \ni (t,\zeta) \to \Phi(t,\zeta)\stackrel{def}{=}
\widetilde{\varphi}_t(\zeta)\in \widetilde{\mathcal{X}}
\end{equation}
provides an isomorphism from the trivial bundle
$\big(\mathbb{R} \times S ,\, {\rm pr}_1\,, \mathbb{R}\times E,\, \mathbb{R}\big)$
to the lifted
bundle $\widetilde{\mathfrak F}_{\sf{g},\sf{m}}$. Define by $\varphi_t$ the projection
\begin{equation}\label{eqEl5}
\tilde{\mathcal{X}}\supset \widetilde{\mathcal P}^{-1} (t)
\ni \widetilde{\varphi}_t(\zeta) \overset{{\tilde p}}{-\!\!\!\longrightarrow} \varphi_t(\zeta)= \tilde{p}(\widetilde{\varphi}_t(\zeta))\in {\mathcal P}^{-1}(e^{2\pi i t}) \in \mathcal{X}\,.
\end{equation}
Since $\widetilde{\mathcal{P}}^{-1}(t)={\mathcal{P}}^{-1}(e^{2\pi i t}) $
and $\tilde{E}_t=E_{e^{{2\pi i t}}}$,
we obtain a smooth family
of homeomorphisms $\varphi_t : {\mathcal P}^{-1} (1) \to {\mathcal
P}^{-1} (e^{2\pi it})$ that map distinguished points to distinguished points.
We call the family $\varphi_t, \,t \in \mathbb{R}, $ a trivializing family of homeomorphisms.
Consider the time-$1$ map
\begin{equation}\label{eqEl4}
\varphi_1 : {\mathcal P}^{-1} (1) \to {\mathcal P}^{-1} (e^{2\pi i})
= {\mathcal P}^{-1} (1) \,
\end{equation}
which is a  self-homeomorphism of the fiber over $1$ that maps the set of distinguished points $E\subset S\cong {\mathcal P}^{-1} (1)$ to itself.
For $k \in \mathbb{Z}$ we have $\varphi_{t+k}=\varphi_t \circ \varphi_1^k.$
For each $n \in \mathbb{Z}$ the projection \eqref{eqEl5}
maps the points $\Phi(t,\zeta)=\widetilde{\varphi}_t(\zeta)$ and $\Phi(t+n,
{\varphi}_1^{-n}(\zeta))=\widetilde{\varphi}_{t+n}(\varphi_1^{-n}(\zeta))$
to the same point $ \varphi_t(\zeta) \in {\mathcal P}^{-1} (e^{2\pi i t})$.
The group ${\mathbb Z}$ of integer numbers acts on $\Phi({\mathbb R}
\times S) = \widetilde{\mathcal X}$ by
\begin{equation}\label{eqEl6}
\Phi({\mathbb R} \times S) \ni \Phi (t,\zeta) \to
\Phi(t+n,\varphi_1^{-n}(\zeta)) \, , \quad n \in {\mathbb Z} \, .
\end{equation}
The quotient
\begin{equation}\label{eqEl6a}
{\mathcal X}_1 \stackrel{def}{=} {\mathbb R} \times S \diagup \big((t,\zeta) \sim (t+1,
\varphi_1^{-1}(\zeta))\big)
\end{equation}
is a smooth manifold that is diffeomorphic to $\mathcal{X}$.
Indeed, the mapping ${{\tilde p}}\circ\Phi: \mathbb{R}\times S\to \mathcal{X}$,
${\tilde{p}}\circ\Phi(t,\zeta)= \varphi_t(\zeta)\in {\mathcal P}^{-1} (e^{2\pi i t})$, satisfies the condition ${\tilde{p}}\circ\Phi\big(t+1,\varphi_1^{-1}(\zeta)\big)={\tilde{p}}\circ\Phi(t,\zeta)$. Hence, this mapping descends to a mapping from $\mathcal{X}_1$ to $\mathcal{X}$, that is a local diffeomorphism and is bijective by the construction. Hence, it is a diffeomorphism from $\mathcal{X}_1$ onto $\mathcal{X}$.
Put
\begin{equation}\label{eqEl6b}
\mathbold{E}_1\stackrel{def}= \mathbb{R}\times E\diagup \Big((t,\zeta) \sim (t+1,
\varphi_1^{-1}(\zeta))\Big)\,,
\end{equation}
and recall that ${\tilde p}(\Phi(t,E))={\tilde p}(\widetilde{\varphi}_t(E))= \mathbold{E}\cap \mathcal{P}^{-1}(e^{2\pi i t})$, since $\widetilde{\varphi}_t$ maps the distinguished points $E\subset S= \mathcal{P}^{-1}(1)$ to the distinguished points $\mathbold{E}\cap \mathcal{P}^{-1}(e^{2\pi i t})$ in the fiber $\widetilde{\mathcal{P}}^{-1}(t)=\mathcal{P}^{-1}(e^{2\pi i t})$.
Define the projection $\mathcal{P}_1:{\mathcal X}_1 \to
\partial \mathbb{D}$, so that $\mathcal{P}_1$ takes the value $e^{2\pi it}$ on the class
containing $(t,\zeta)$. We obtain a bundle $({\mathcal X}_1,\mathcal{P}_1,\mathbold{E}_1, \partial\mathbb{D})$ that is smoothly isomorphic to
$({\mathcal X},\mathcal{P},\mathbold{E}, \partial \mathbb{D})$.
We will also say for short that
${\mathfrak F}_{\sf{g},\sf{m}}$ is smoothly isomorphic to the
mapping torus
\index{mapping ! mapping torus}
\begin{equation}\label{eqEl7}
\left([0,1] \times S\right) \diagup \Big((0 , \zeta) \sim
(1,\varphi (\zeta)) \Big)\,
\end{equation}
where $\varphi = \varphi_1^{-1}$ is a smooth orientation preserving
self-homeomorphism of the fiber $S=\mathcal{P}^{-1}(1)$
with set of distinguished points $E\subset S$.
The mapping $\varphi$ depends
on the trivializing vector field. However, the mapping class of $\varphi$ is independent on
this vector field, it is merely determined by the bundle.
We denote it by $\mathfrak{m}_{\mathfrak{F}}$ (with $\mathfrak{F}={\mathfrak
F}_{\sf{g},\sf{m}}$).
The mapping class $\mathfrak{m}_{\mathfrak{F}} \in \mathfrak{M}(S;\emptyset, E)$
is called the monodromy mapping class
of the bundle ${\mathfrak F}$ over the circle, or the monodromy, for short. \index{monodromy}

For diffeomorphic closed surfaces  $S_1$ and $S_2$ and sets of distinguished points $E_1$ and $E_2$ of the same cardinality the groups $\mathfrak{M}(S_1;\emptyset, E_1)$ and $\mathfrak{M}(S_2;\emptyset, E_2)$ are isomorphic.
Any diffeomorphism $\varphi:S_1\to S_2\,$ with $\varphi(E_1)=E_2\,$ induces an isomorphism ${\rm Is}_{\varphi}:\mathfrak{M}(S_1;\emptyset,E_1)\to \mathfrak{M}(S_2;\emptyset,E_2)$. Two such isomorphisms $\mathfrak{M}(S_1;\emptyset,E_1)\to \mathfrak{M}(S_2;\emptyset,E_2)$ differ
by conjugation with an element of $\mathfrak{M}(S_2;\emptyset, E_2)$. Hence, there is a canonical one-to-one correspondence between the sets of conjugacy classes $\reallywidehat{\mathfrak{M}(S_1;\emptyset,E_1)}$ and $\reallywidehat{\mathfrak{M}(S_2;\emptyset,E_2)}$, and we will identify conjugacy classes according to this correspondence. Similarly, let $h:F_k\to \mathfrak{M}(S_1;\emptyset, E_1)$ be a homomorphism from a free group $F_k$ of rank $k$. Any diffeomorphism $\varphi:S_1\to S_2\,$ with $\varphi(E_1)=E_2\,$ induces another homomorphism $h_{\varphi}:
F_k\to \mathfrak{M}(S_2;\emptyset,E_2)$.
Replacing $\varphi$ by another diffeomorphism we arrive at a conjugate homomorphism.
We obtain a canonical bijection between conjugacy classes of homomorphisms $F_k\to \mathfrak{M}(S_j;\emptyset,E_j), \, j=1,2$.

\begin{probl}\label{problEl.1}
Prove that for each isotopy class $\widehat{\mathfrak{F}}_{\sf{g},\sf{m}}$ of   $(\sf{g},\sf{m})$-bundles over the circle and the associated conjugacy class $\widehat{\mathfrak{m}}_{{\mathfrak{F}}_{\sf{g},\sf{m}}}$ of elements of the mapping class of the fiber over the point $1$ the equality
\begin{align*}
\mathcal{M}(\widehat{\mathfrak{F}}_{\sf{g},\sf{m}})=\frac{\pi}{2} \frac{1}{h(\widehat{\mathfrak{m}}_{{\mathfrak{F}}_{\sf{g},\sf{m}}})}
\end{align*}
holds.
\end{probl}

In the sequel we will use the mapping class group $\mathfrak{M}(S;\emptyset,E)$ on a reference (Riemann) surface of type $(\sf{g},\sf{m})$ and the set of its
conjugacy classes $\reallywidehat{\mathfrak{M}(S;\emptyset,E)}$, or the set of conjugacy classes of mappings from a free group to $\mathfrak{M}(S;\emptyset,E)$, respectively. Recall that the group $\mathfrak{M}(S;\emptyset,E)$ is isomorphic to the modular group ${\rm Mod}(\sf{g},\sf{m})$.

The
following theorem holds.
(See, e.g. \cite{FaMa} for the case of closed fibers of genus $g\geq 2$ and paracompact Hausdorff spaces $X$.)
\begin{thm}\label{thmEl1}
Let $X$ be a connected smooth finite open oriented surface with base point $q_0$. The set of isomorphism classes \index{isotopy class ! of fiber
bundles}of smooth oriented $({\sf g},{\sf m})$-bundles on $X$ is in
one-to-one correspondence to the set of conjugacy classes of
homomorphisms from the fundamental group $\pi_1 (X,q_0)$ into the
modular group ${\rm Mod}(\sf{g},\sf{m})$.
\end{thm}
{\bf Proof.} Let first $\mathfrak{F}_0=(\mathcal{X}_0,\mathcal{P}_0,\mathbold{E}_0,\partial \mathbb{D})$ and $\mathfrak{F}_1=(\mathcal{X}_1,\mathcal{P}_1,\mathbold{E}_1,\partial \mathbb{D})$ be isomorphic bundles over the circle $\partial\mathbb{D}$. Suppose the isomorphism is given by a diffeomorphism $\phi: \partial \mathbb{D}\toitself$ and a diffeomorphism $\Phi:\mathcal{X}_0\to\mathcal{X}_1$. We may assume  that $\phi$ preserves the base point $1$, hence, $\Phi$ maps each fiber $\mathcal{P}_0^{-1}(e^{2\pi i t})$ of the first bundle onto the fiber $\mathcal{P}_1^{-1}(e^{2\pi i \phi_1(t)})$ of the second bundle for a homeomorphism $\phi_1:[0,1]\toitself$.
Let $\varphi_t:S\to \mathcal{P}_0^{-1}(e^{2\pi i t}) $ be a trivialising family of homeomorphisms for the bundle $\mathfrak{F}_0$. Then $\psi_t\stackrel{def}= \Phi|\mathcal{P}_0^{-1}(e^{2\pi i t}) \circ \varphi_t \circ (\Phi|\mathcal{P}_0^{-1}(1))^{-1}$ is a trivialising family of homeomorphisms of the second bundle. We obtain
$\psi_1=\Phi|\mathcal{P}_0^{-1}(1) \circ \varphi_1 \circ (\Phi|\mathcal{P}_0^{-1}(1))^{-1}$. This means that
$\psi_1$ and $\phi_1$ represent the same element of $\reallywidehat{\mathfrak{M}(\mathcal{P}_0^{-1}(1))}$.  We saw that the monodromies of isomorphic $(\sf{g},\sf{m})$-bundles over the circle belong to the same conjugacy class  $\reallywidehat{{\rm Mod}(\sf{g},\sf{m})}$.

Vice versa, suppose the monodromies of two $(\sf{g,m})$-bundles $\mathfrak{F}_0$ and $\mathfrak{F}_1$ over the circle
with base point $1$ represent the same element of $\reallywidehat{{\rm Mod}(\sf{g},\sf{m})}$. We will prove that  the bundles are isomorphic.
The assumption means that the monodromies $\mathfrak{m}_j$ of $\mathfrak{F}_j, \, j=0,1$ are represented by mappings
$\varphi_j \in \mathfrak{m}_j$, such that $\varphi_1=\varphi \circ \varphi_0\circ\varphi^{-1}$ for a homeomorphism
$\varphi:\mathcal{P}_0^{-1}(1)\to \mathcal{P}_1^{-1}(1)$ that maps the set of distinguished points of $\mathcal{P}_0^{-1}(1)$ onto the set of distinguished points of $\mathcal{P}_1^{-1}(1)$.
The bundle $\mathfrak{F}_0$ is isomorphic to
the mapping torus of $\varphi_0$. This mapping torus has total space $\mathcal{X}_0=\big(\mathbb{R}\times S\big)\diagup \big((t,\zeta)\sim (t+1,\varphi_0(\zeta))\big)$ and set of distinguished points
$E_0=\big(\mathbb{R}\times E\big)\diagup \big((t,\zeta)\sim (t+1,\varphi_0(\zeta))\big)$.

The mapping torus $\mathfrak{F}_0'$ with total space
$\mathcal{X}_0'=\big(\mathbb{R}\times \varphi(S)\big)\diagup \big((t,\zeta)\sim (t+1,\varphi\circ \varphi_0\circ \varphi^{-1} (\zeta))\big)$ and set of distinguished points $E_0'=\big(\mathbb{R}\times \varphi(E)\big)\diagup \big((t,\zeta)\sim (t+1,\varphi \circ \varphi_0 \circ \varphi^{-1} (\zeta))\big)$ is isomorphic to $\mathfrak{F}_0$ and has monodromy represented by $\varphi\circ \varphi_0\circ \varphi^{-1} $.
Indeed, the mapping
$\mathbb{R}\times S \ni(t,\zeta)\to (t,\varphi(\zeta)) \in \mathbb{R}\times S$ descends to a diffeomorphism $\mathcal{X}_0\to \mathcal{X}_0'$ that takes fibers to fibers.

The mapping torus $\mathfrak{F}_0'$ with total space $\mathcal{X}_0'$ is isomorphic to $\mathfrak{F}_1$. Indeed, the monodromy mappings of  $\mathfrak{F}_0'$ and $\mathfrak{F}_1$ coincide, and the mapping tori of isotopic mappings are isotopic, and, hence, isomorphic. Since the mapping torus $\mathfrak{F}_0'$ is also isomorphic to $\mathfrak{F}_0$ , the mapping tori $\mathfrak{F}_0$ and $\mathfrak{F}_1$ are isomorphic. Hence, $\mathfrak{F}_0$ and $\mathfrak{F}_1$ are isomorphic.

\smallskip

Let now $X$ be a connected smooth  oriented open surface of genus $g$ with $m$ holes with base point $q_0$. Suppose $\mathfrak{F}=(\mathcal{X},\mathcal{P}, \mathbold{E},X)$ is a smooth  $({\sf g},{\sf m})$-bundle over $X$. Denote the fiber $\mathcal{P}^{-1}(q_0)$ over $q_0$ by $S$ and the set of distinguished points $\mathbold{E}\cap \mathcal{P}^{-1}(q_0)$ by $E$.
We assign to $\mathfrak{F}$ a homomorphism $\pi_1(X,q_0)\to \mathfrak{M}(S, \emptyset, E)$ as follows.
Take smooth closed curves in $X$ parameterized by $\gamma_j:[0,1]\to X$ that represent the elements $e_j$ of a standard system of generators of  the fundamental group $\pi_1(X,q_0)$, $j=1,\ldots, 2g+m-1$. Associate to each generator $e_j$ of $\pi_1(X,q_0)$ the monodromy mapping class of the restricted bundle $\mathfrak{F}|{{\gamma_j}} $. This is
a mapping class on the fiber $S=\mathcal{P}^{-1}(q_0)$ over $q_0$ with distinguished points $E=\mathbold{E}\cap \mathcal{P}^{-1}(q_0)$, that depends only on the $e_j$ and on
the bundle.
We obtain a well-defined mapping that associates to each generator of the fundamental group $\pi_1(X,q_0)$ a mapping class. This mapping extends to a homomorphism
from the fundamental group $\pi_1(X,q_0)$ to the mapping class group $\mathfrak{M}(S, \emptyset, E)$.
In the same way as in the case of bundles over the circle we may prove that the respective homomorphisms corresponding to isomorphic bundles represent the same conjugacy class of homomorphisms
from $\pi_1(X,q_0)$ to ${\rm Mod}(\sf{g,m})$.

Vice versa, take a homomorphism $h$ from the fundamental group $\pi_1(X,q_0)$ to the modular group ${\rm Mod}(\sf{g,m})$.
A bundle over $X$
whose monodromy homomorphism coincides with $h$ can be obtained as follows.

By Lemma \ref{lemEl2b} the surface $X$ is diffeomorphic to a standard neighbourhood of a standard bouquet $B$ of circles for $X$. Hence, it can be written as union $X=D\cup \bigcup_j V_j$ of an open disc $D$ and half-open bands $V_j$ attached to $D$.  The disc $D$ contains the base point $q_0$. The boundary $\partial D$ of $D$ in $X$ is smooth.
The set $B\setminus D$ is the union of disjoint closed arcs  ${\sf s}_j$ with endpoints on $\partial D$. Each circle $c_j$ of the bouquet contains exactly one of the  ${\sf s}_j$.
For each $j$ the set $V_j$ is a neighbourhood in $X\setminus D$ of  ${\sf s}_j$.

Consider the universal covering $\tilde{X} \overset{{ \sf P}}{-\!\!\!\longrightarrow} X$.
Take a point $\tilde{q}_0\in \tilde{X}$ with ${\sf P}(\tilde{q}_0)=q_0$.
For each $j$ we take the lifts $\widetilde{c}_j$ of $c_j$ to $\tilde X$ with initial point $\tilde{q}_0$, and the lift $\widetilde{V}_j$ of $V_j$
that intersects $\widetilde{c}_j$.
Let $\widetilde{D}_0$ be the lift of $D$ to $\tilde X$, that contains $\tilde{q}_0$, and  let $\widetilde{D}_j,\, j=1,\ldots,2{g}+{m}-1, $ be the lift of $D$, that contains the endpoint $\tilde{x}_j\stackrel{def}=\widetilde{c}_j(1)$ of $\widetilde{c}_j$. Let $\tilde U$ be the domain in $\tilde X$ that is the union of $\tilde{D}_j,\, j=0,1,\ldots, 2g+m-1,$ and $\tilde{V}_j,\, j=1,\ldots,2g+m-1$.
\index{$D$} \index{$V_j$}
\index{$\widetilde{D}_0$} \index{$\widetilde{D}_j$} \index{$\widetilde{V}_j$}

Choose for each $j\geq 1$ a mapping $\varphi_j$ in the mapping class $h(e_j)$, and define a bundle over $X$ by taking the trivial bundle over $\tilde U$ and
making the following identifications. For each $x'\in D$ and each $j$ we glue the  fiber over the point $\tilde{x}'_j\in \widetilde{D}_j$ for which $p(\tilde{x}'_j)=x'$ to the fiber over the point $\tilde{x}'_0\in \widetilde{D}_0$ for which $p(\tilde{x}'_0)=x'$  using the mapping $\varphi_j$. We obtained a bundle with the given monodromy homomorphism.

As in the case of bundles over the circle one can see that each bundle is isomorphic to a bundle that is constructed in the just described way, and bundles constructed in this way with equal conjugacy class of monodromy homomorphisms are isomorphic. We obtained a bijective correspondence of isomorphism classes of smooth $(\sf{g,m})$-bundles over a connected finite open  smooth oriented surface $X$ of genus $g$ with $m>0$ holes and conjugacy classes of homomorphisms from the fundamental group of $X$ to
 ${\rm Mod}(\sf{g,m})$. \hfill $\Box$

\smallskip

\noindent {\bf Isotopy classes and ismorphism classes of bundles.}
We will consider now the relation between isotopy classes and isomorphism classes of $({\sf g,m})$-bundles, and their relation to the monodromies of the bundles.

Lemma \ref{lemEl2} says that isotopic smooth $({\sf g,m})$-bundles over the circle are smoothly isomorphic. The following lemma states in particular, that vice versa isomorphic bundles are isotopic.
\begin{lemm}\label{lemEl2a}
Two  smooth $({\sf g,m})$-bundles over a connected smooth finite open oriented surface $X$ are isomorphic if and only of they are isotopic.
\end{lemm}
{\bf Proof.}  By Lemma \ref{lemEl2} and Theorem \ref{thmEl1} it remains to prove that two
 smooth $({\sf g,m})$-bundles over $X$ are isotopic if their monodromy homomorphisms represent the same conjugacy class of homomorphisms.

We first consider two smooth $({\sf g,m})$-bundles $\mathfrak{F}_0$ and $\mathfrak{F}_1$ over the circle
with fiber and distinguished points over the base point $1$ being equal. Suppose  the bundles $({\sf g,m})$-bundles $\mathfrak{F}_0$ and $\mathfrak{F}_1$  have the same monodromy. We will prove that they are isotopic by an isotopy that fixes the fiber over the base point and the set of distinguished points in this fiber (by a based isotopy, for short). Denote the fiber over $1$ by $S$ and the set of distinguished points in $S$ by $E$. Lift each bundle $\mathfrak{F}_j,\,j=0,1,$ to a bundle over the vertical line $\{j\}\times\mathbb{R} \subset \mathbb{R}^2$ by the covering $\{j\}\times \mathbb{R}\ni (j,t)\to e^{2\pi i t}$.
Let $\varepsilon$ be a small positive number.
Restrict the lift of the bundle $\mathfrak{F}_j$ to the open interval $\{j\}\times (-\varepsilon,1+\varepsilon)\supset\{j\}\times [0,1]$, $\varepsilon>0$, and extend for each $j$ the restricted bundle to a bundle over the rectangle $(j-\varepsilon,j+\varepsilon)\times (-\varepsilon,1+\varepsilon)$, so that the restriction of the extended bundle $\mathfrak{F}_j'$ to each horizontal segment is a product bundle.

Consider the product bundle with fiber $S$ and set of distinguished points $E$ over the horizontal line segments $(-\varepsilon,1+\varepsilon)\times \{k\}$, $k=0,1$. Take an extension of the product bundle over each of the horizontal line segments $ (-\varepsilon,1+\varepsilon)\times\{k\}$ to  a smooth $({\sf g,m})$-bundle over  the open rectangle $ (-\varepsilon,1+\varepsilon)\times(k-\varepsilon,k+ \varepsilon)$  around the segment, so that for $y \in (-\varepsilon,\varepsilon)$ and $x\in (-\varepsilon,1+\varepsilon)$ the fibers and the distinguished points over $(x,y)$ and $(x,y+1)$ are the same, and
the bundles over the four rectangles   $(-\varepsilon,1+\varepsilon)\times (k-\varepsilon,k+ \varepsilon)$, $k=0,1$ and  $(j-\varepsilon,j+\varepsilon)\times (-\varepsilon,1+\varepsilon)$, $j=0,1,$  match to a smooth bundle $\tilde{\mathfrak{F}}_0$ over the neighbourhood $V\stackrel{def}=\Big((-\varepsilon, 1+\varepsilon)\times (-\varepsilon, 1+\varepsilon)\Big) \, \setminus \,\Big((\varepsilon, 1-\varepsilon)\times (\varepsilon, 1-\varepsilon)\Big) $  of the boundary of the square $[0,1]  \times[0,1]$. Let $(0,0)$ be the base point of $V$. The monodromy of the obtained bundle along the generator of the fundamental group $\pi_1(V,0)$ of $V$
is the identity. Hence, the bundle is smoothly isomorphic to the trivial bundle.

Let $\varphi$ be a diffeomorphism of the total space of the bundle  $\tilde{\mathfrak{F}}_0$ to the total space of the trivial bundle over $V$, that maps the fiber of the bundle $\tilde{\mathfrak{F}}_0$ over each point in $V$ to the fiber of the trivial bundle over this point, and maps distinguished points to distinguished points.
The trivial bundle over $V$
extends to the trivial bundle over the neighbourhood $(-\varepsilon, 1+\varepsilon)\times (-\varepsilon, 1+\varepsilon)$ of the square. Glue the
trivial bundle over the square $(\frac{\varepsilon}{2}, 1-\frac{\varepsilon}{2})\times (\frac{\varepsilon}{2}, 1-\frac{\varepsilon}{2})$, that is relatively compact in the open unit square, to the bundle $\tilde{\mathfrak{F}}_0$ over $V$ along $V\cap  \big((\frac{\varepsilon}{2}, 1-\frac{\varepsilon}{2})\times (\frac{\varepsilon}{2}, 1-\frac{\varepsilon}{2})\big)$
using the mapping $\varphi$. We obtain a smooth bundle over
the neighbourhood $(-\varepsilon, 1+\varepsilon)\times (-\varepsilon, 1+\varepsilon)$ of the unit square, that coincides with $\tilde{\mathfrak{F}}_0$ over $V$.
For $(x,y)\in (-\varepsilon,1+\varepsilon)\times (-\varepsilon, \varepsilon)$ we glue the fiber over $(x,y)$ to the fiber over $(x,y+1)$
using the identity mapping. We get a bundle over the product of the interval $(-\varepsilon, 1+\varepsilon)$ with the circle, which provides
an isotopy with fixed fiber over the base point and fixed set of distinguished points (a based isotopy for short).

If two bundles over the circle with the same fiber over the point $1$ and the same set of distinguished points over the base point have conjugate monodromy, we consider instead of the product bundle on each of the two horizontal line segments $(-\varepsilon,1+\varepsilon)\times \{k\}$, $k=0,1$, the bundle obtained from the product bundle by the mapping $\Phi,\, \Phi(t,\zeta)=\varphi_t(\zeta), \, t\in (-\varepsilon,1+\varepsilon), \zeta \in S,$ between total spaces, where $\varphi_t$ is a smooth family of diffeomorphisms with $\varphi_0={\rm Id}$ and $\varphi_1$ representing the conjugating mapping class.
Along the same lines as above one can prove that the bundles are free isotopic.

If two bundles over the circle have different fibers $S_1$ and $S_2$ over the point $1$ with sets of distinguished points $E_1$ and $E_2$, respectively, and (after an isomorphism) their monodromy maps are conjugate, then the bundles are free isotopic. Indeed, the condition means that the monodromies of the two bundles differ by ''conjugation'' by a homeomorphism $\psi:S_1\to S_2$ that maps $E_1$ onto $E_2$. Instead of the product bundle on the horizontal segments $(-\varepsilon,1+\varepsilon)\times \{k\}$, $k=0,1$, we use in this case the bundle obtained from the product bundle by the mapping $\Phi',\, \Phi'(t,\zeta)=\varphi'_t(\zeta), \, t\in (-\varepsilon,1+\varepsilon), \zeta \in S,$ between total spaces, where $\varphi'_t$ is a smooth family of diffeomorphisms with $\varphi'_0={\rm Id}$ and $\varphi'_1=\psi$.

The proof for bundles over a smooth finite open oriented surface $X$ follows along the same lines using the fact that such a surface $X$ is diffeomorphic to a standard neighbourhood of a standard bouquet of circles.

We obtain a bijective correspondence between smooth isomorphism classes and free isotopy classes of $({\sf g,m})$-bundles over connected oriented smooth finite open surfaces. \hfill $\Box$

\medskip

A smooth $(\sf{g},\sf{m})$-bundle $(\mathcal{X},\mathcal{P}, \mathbold{E},X)$ over a connected smooth finite open oriented  surface $X$ admits a smooth section, i.e. a smooth mapping $X\ni x \to s(x)\in \mathcal{X}\setminus\mathbold{E} $ for which $\mathcal{P}\circ s(x)=x$ for all $x\in X$.
Indeed, let $N(B)$ be a neighbourhood of a standard bouquet of circles for $X$ as in Section \ref{sec:2.0},  and let $D$, $V_j$ and ${\sf s}_j$ be as in Section \ref{sec:2.0}.
There is a smooth section over
$D\cup_j{\sf s}'_j$, where ${\sf s}'_j$
are the arcs of $\partial D$ to which the horizontal sides of the $V_j$ are glued. This follows from the fact, that there is a simply connected neighbourhood of this set, and the bundle is smoothly trivial over any simply connected domain.
Extend the smooth section to a smooth section over a neighbourhood of $D\cup_j({\sf s}'_j\cup  {\sf s}_j)$.
Since a neighbourhood of each $V_j$ is simply connected and there is a smooth deformation retraction of a neighbourhood of $V_j$ to a neighbourhood of ${\sf s}'_j\cup  {\sf s}_j$,  we obtain a smooth section on the whole $N(B)$. By Lemma \ref{lemEl2b} we obtain a smooth

\medskip

\noindent {\bf Reducible bundles and irreducible components.}
Thurston's notion of reducible mapping classes takes over to families of mapping classes on a surface of type $({\sf{g}},{\sf{m}})$, and therefore to $({\sf{g}},{\sf{m}})$-bundles. Namely,
an admissible system of
curves on a connected oriented  closed surface $S$ of genus $\sf g$ with set of $\sf m$ distinguished points $E$ is said to reduce a family of mapping classes $\mathfrak{m}_j \in \mathfrak{M}(S; \emptyset, E) $ if it reduces each $\mathfrak{m}_j$. If the family consists of a single mapping class we arrive at the notion of a reduced mapping class (see Section \ref{sec:2.4}).

Similarly, a $({\sf{g}},{\sf{m}})$-bundle over a finite open connected oriented  surface $X$ with fiber $S$ over the base point $q_0$ and set of distinguished points $E\subset S$ is called reducible if there is an admissible system of curves in the fiber over the base point
that reduces all monodromy mapping classes simultaneously. Otherwise the bundle is called
irreducible.

Consider any isotopy class of reducible $(\sf{g},\sf{m})$-bundles with fiber $S$ and set of distinguished points $E\subset S$ over the
base point $q_0$. Take an admissible system $\mathcal{C}$ of simple closed curves on $S\setminus E$
that reduces each monodromy mapping class and is maximal in the sense that there is no
strictly larger admissible system with this property. For each monodromy mapping class
$\mathfrak{m}_j$ we choose a representing homeomorphism $\varphi_j$ that fixes $\mathcal{C}$
setwise. Then all homeomorphisms $\varphi_j$ fix $S\setminus \mathcal{C}$. Decompose
$S\setminus \mathcal{C}$ into disjoint open sets $S_k$, each of which being a minimal union
of connected components of $S\setminus \mathcal{C}$ that is invariant under each
$\varphi_j$.
Each $S_k$ is a union of surfaces with holes and possibly with distinguished points and is homeomorphic to a union
$\overset{\circ}{S}_k$ of closed surfaces with a set $E_k$ of finitely many punctures and, possibly, a set $E_k'$ of finitely many distinguished points. The restrictions of the $\varphi_j$ to
each $S_k$ (possibly, with distinguished points) are conjugate to self-homeomorphisms of $\overset{\circ}{S}_k$ (possibly with set of distinguished points $E_k'$).

For each $k$ we denote by $S_k^c$ the (possibly non-connected) closed Riemann surface which is obtained from $\overset{\circ}{S}_k$ by filling the punctures.
The Riemann surface $S_k^c$ is equipped with distinguished points $E_k\cup E'_k$. For each $k$ we obtain
a conjugacy class of homomorphisms from the fundamental group $\pi_1(X,x_0)$
to the modular group of the possibly non-connected Riemann surface $S_k^c$ with set of distinguished points $E_k\cup E'_k$. By Theorem \ref{thmEl1} (more precisely, by its analog for possibly non-connected fibers) we obtain for each $k$ an isotopy class of bundles over
$X$ with fiber a union of closed surfaces with distinguished points. The obtained isotopy classes of bundles are
irreducible and are called irreducible bundle components of the  $(\sf{g},\sf{m})$-bundle
with fiber $S$. (There may be different decompositions into irreducible bundle components.)
The reducible bundle can be recovered from the irreducible bundle components up to commuting Dehn twists
in the fiber over the base point. (See the proof of Theorem  \ref{thmEl.0} below for the case of $(0,4)$-bundles.)

\noindent {\bf Locally holomorphically trivial and isotrivial bundles.}
Let $X$ be a Riemann surface and $2{\sf g}-2 + {\sf m}>0$.
A holomorphic $({\sf g},{\sf m})$-bundle over $X$ is called locally holomorphically trivial if it is locally holomorphically isomorphic to the trivial  $({\sf g},{\sf m})$-bundle. All fibers of a locally holomorphically trivial $({\sf g},{\sf m})$-bundle ${\mathfrak{F}}=({\mathcal{X}}, {\mathcal{P}}, {\mathbold{E}}, {X})$
are conformally equivalent to each other.

For a locally trivial holomorphic $({\sf g},{\sf m})$-bundle there exists a finite unramified covering $\hat{\sf P} :\hat{X}\to X$ and a lift
$\hat{\mathfrak{F}}=(\hat{\mathcal{X}}, \hat{\mathcal{P}}, \hat{\mathbold{E}}, \hat{X})$
of $\mathfrak{F}$ to $\hat X$ such that $\hat{\mathfrak{F}}$ is holomorphically isomorphic to the trivial bundle. This can be seen as follows. Consider the lift $\tilde{\mathfrak{F}}$ of the bundle $\mathfrak{F}$ to the universal covering ${\sf P}:\tilde X\to X$ of $X$, i.e.
$\tilde{\mathfrak{F}}=(\tilde{\mathcal{X}}, \tilde{\mathcal{P}}, \tilde{\mathbold{E}}, \tilde{X})$, where the fiber $\tilde{\mathcal{P}}^{-1}(\tilde{x})$ with distinguished points $\tilde{\mathbold{E}}\cap \tilde{\mathcal{P}}^{-1}(\tilde{x})$ is conformally equivalent to the fiber ${\mathcal{P}}^{-1}({x})$ with distinguished points ${\mathbold{E}}\cap {\mathcal{P}}^{-1}({x})$ with $x={\sf P} (\tilde{x})$. Let $\tilde{P}:\tilde{\mathcal{X}}\to \mathcal{X}$ be the respective fiber preserving projection. The bundle $\tilde{\mathfrak{F}}$ is locally holomorphically trivial. Since $\tilde X$ is simply connected, $\tilde{\mathfrak{F}}$ is holomorphically trivial on $\tilde X$, hence,  there is a biholomorphic mapping $\Phi: \tilde{\mathcal{X}}\to \tilde{X}\times S$ that maps $\tilde{\mathcal{P}}^{-1}(\tilde{x})$ to $\{\tilde{x}\}\times S$ for each $\tilde{x}\in\tilde{X}$ and maps $\tilde{\mathbold{E}}$ to $\tilde{X}\times E$.
Here $S$ is the fiber $\tilde{\mathcal{P}}^{-1}(\tilde{q}_0)\cong {\mathcal{P}}^{-1}({q}_0)$ over a chosen point $\tilde{q_0}\in \tilde{X}$ over the base point $q_0\in X$ and $E= \tilde{\mathbold{E}}\cap {\mathcal{P}}^{-1}(\tilde{x})$. The mapping $\Phi^{-1}$ provides a
holomorphic family of conformal mappings $ \tilde\varphi_{\tilde{x}}: S=\tilde{\mathcal{P}}^{-1}(\tilde{q}_0) \to \tilde{\mathcal{P}}^{-1}(\tilde{x})$, $\tilde{x}\in \tilde{X}$. Each mapping of the family takes the set of distinguished points in the fiber $S$ to  the set of distinguished points in the respective fiber of the bundle $\tilde{\mathfrak{F}}$.
The total space ${\mathcal{X}}$ of the bundle $\mathfrak{F}$ is holomorphically equivalent to the quotient of $\{\tilde{x}\}\times S$
by the following equivalence relation.
Two points $(\tilde{x}_1,\zeta_1)$ and $(\tilde{x}_2,\zeta_2)$ in $\tilde{X}\times S$ are equivalent if $\tilde{P}\Phi^{-1}(\tilde{x}_1,\zeta_1)=\tilde{P}\Phi^{-1}(\tilde{x}_2,\zeta_2)$, i.e. if
${\sf P}(\tilde x_1)={\sf P}(\tilde x_2)$ and
$(\tilde{P}\varphi_{\tilde{x}_1})(\zeta_1)=(\tilde{P}\varphi_{\tilde{x}_2})(\zeta_2)$.
Put $\varphi_{\tilde{x}}=    \tilde{P}\tilde\varphi_{\tilde{x}}:S\to  {\mathcal{P}}^{-1}(x)$ with $x={\sf P}(\tilde x)$.
If  ${\sf P}(x_1)={\sf P}(x_2)$, the mapping $\varphi_{\tilde{x}_2}^{-1}
\varphi_{\tilde{x}_1}$
is a holomorphic self-homeomorphism of the fiber $S$.
The set of such self-homeomorphims is finite.
\index{fiber bundle ! holomorphically trivial} \index{fiber bundle ! locally holomorphically trivial}
\index{fiber bundle ! isotrivial}

Let $\tilde{x}= \sigma(\tilde q_0)$ for a covering transformation $\sigma$. The mapping  $\varphi_{\sigma(\tilde q_0)}$ maps $S$ to $\mathcal{P}^{-1}(q_0)\cong S$. For two covering transformations $\sigma_1$ and $\sigma_2$ the equality
\begin{align}\label{eqEl17'}
\varphi_{\sigma_1 \sigma_2(\tilde q_0)}=\varphi_{\sigma_2(\tilde q_0)} \varphi_{\sigma_1(\tilde q_0)}
\end{align}
holds.

As before  $({\rm Is}^{\tilde{q}_0})^{-1}$ denotes the isomorphism from the fundamental group to the group of covering transformations.
Consider the set $N$ of elements $e\in \pi_1(X,q_0)$ for which $\varphi_{({\rm Is}^{\tilde{q}_0})^{-1}(e)(\tilde{q}_0)}$
is the identity. By \eqref{eqEl17'} the set $N$ is a normal subgroup of the fundamental group. It is of finite index, since two cosets $e_1 \,N$ and $e_2\, N$ are equal if
$\tilde{P}\varphi_{({\rm Is}^{\tilde{q}_0})^{-1}(e_2 e_1^{-1})(\tilde{q}_0)}={\rm Id}$,
and there are only finitely many distinct holomorphic self-homeomorphisms of $S\cong {\mathcal{P}}^{-1}({q}_0)$.
Hence, $\hat{X}\stackrel{def}=\tilde{X}\diagup ({\rm Is}^{\tilde{q}_0})^{-1}(N)$
is a finite unramified covering of $X$ and the lift of the bundle $\mathfrak{F}$ to $\hat{X}$ has the required property.

Vice versa, if for a holomorphic $({\sf g},{\sf m})$-bundle $\mathfrak{F}$ there exists a finite unramified covering $\hat{\sf P}:\hat{X}\to X$, such that the lift
$\hat{\mathfrak{F}}=(\hat{\mathcal{X}}, \hat{\mathcal{P}}, \hat{\mathbold{E}}, \hat{X})$
of $\mathfrak{F}$ to $\hat X$
is holomorphically isomorphic to the trivial bundle, then $\mathfrak{F}$ is locally holomorphically trivial.

A smooth (holomorphic, respectively) bundle is called isotrivial, if it has a finite covering by the trivial bundle. If all monodromy mapping classes of a smooth bundle are periodic, then the bundle is isotopic (equivalently, smoothly isomorphic) to an
isotrivial bundle. This can be seen by the same arguments as above.

\section{$(0,4)$-bundles over genus $1$ surfaces with a hole}
\label{sec:9.3}
The $(0,n)$-bundles over a manifold $X$ are closely related to the separable quasipolynomials on $X$. In this section we will
state a theorem that concerns $(0,4)$-bundles and is related to Theorem \ref{thmGrom1}. We
will prove it together with Theorem \ref{thmGrom1}.

Let $X$ be a connected finite open Riemann surface  (a connected  oriented finite open smooth surface, respectively).
By a holomorphic (smooth, respectively) $(0,n)$-bundle with a section over $X$ we mean a holomorphic (smooth, respectively) $(0,n+1)$-bundle
$(\mathcal{X},\mathcal{P}, \mathbold{E}, X)$,
such that the complex manifold (smooth manifold, respectively) $\mathbold{E}\subset \mathcal{X}$ is the disjoint union of two complex manifolds (smooth manifolds, respectively) $\mathring{\mathbold{E}}$ and $\mathbold{s}$, where $\mathring{\mathbold{E}}\subset \mathcal{X}$
intersects each fiber $\mathcal{P}^{-1}(x)$ along a set $\mathring{E}_x$ of $n$ points, and $\mathbold{s}\subset \mathcal{X}$ intersects each fiber $\mathcal{P}^{-1}(x)$ along a
single point $s_x$. In other words, the mapping $x\to s_x,\, x\in X$, is a (holomorphic, smooth, respectively) section of the $(0,n)$-bundle with set of distinguished points $\mathring{E}_x$ in the fiber over $x$. Two smooth $(0,n)$-bundles with a section are called
isomorphic if they are isomorphic as smooth $(0,n+1)$-bundles and the diffeomorphism between the total spaces takes the section of one bundle to the section of the other bundle.
Two smooth $(0,n)$-bundles with a section are called isotopic if they are isotopic as $(0,n+1)$-bundles with an isotopy that joins the sections of the bundles.
Two smooth $(0,n)$-bundles with a section are isotopic iff they are isomorphic.
\index{fiber bundle ! $(0,n)$-bundle with a section} \index{fiber bundle ! special $(0,n+1)$-bundle}
\index{$\mathring{\mathbold{E}}$} \index{$\mathbold{s}$}

A special smooth (holomorphic, respectively) $(0,n+1)$-bundle over $X$ is a smooth (holomorphic, respectively) bundle over $X$ of the form $(X\times \mathbb{P}^1, {\rm pr}_1, \mathbold{E}, X)$,
where ${\rm pr}_1 : X\times \mathbb{P}^1 \to X$ is the projection onto the first factor, and the smooth submanifold $\mathbold{E}$ of $X\times \mathbb{P}^1$  is equal to the disjoint union
$\mathring{\mathbold{E}}\cup \mathbold{s}^{\infty}$,
where $\mathring{\mathbold{E}}$ is a smooth submanifold of $X\times \mathbb{P}^1$ that intersects each fiber along $n$ finite points and $\mathbold{s}^{\infty}$ is the smooth submanifold of $X\times \mathbb{P}^1$ that intersects each fiber $\{x\}\times \mathbb{P}^1$ along the point $\{x\}\times \{\infty\}$. A special $(0,n+1)$-bundle is, in particular, a $(0,n)$-bundle with a section.
\index{${\rm pr}_1$} \index{$\mathbold{s}^{\infty}$}

Each smooth mapping $f:X \to C_n ({\mathbb C}) \diagup
{\mathcal S}_n$ from $X$ to the symmetrized configuration space defines a smooth special $(0,n+1)$-bundle over $X$. The smooth submanifold $\mathbold{E}$ of $X\times \mathbb{P}^1$  is equal to the union $\bigcup_{x \in X}\,\big(x,f(x)\cup\{\infty\}\big)$. Vice versa, for each special $(0,n+1)$-bundle the mapping $X\ni x\to \mathring{\mathbold{E}}\cap ({\rm pr}_1)^{-1}(x)$ defines a smooth separable quasi-polynomial $f$ of degree $n$ on $X$.
The special $(0,n+1)$-bundle is holomorphic if and only if the mapping $f$ is
holomorphic.

Choose a base point $x_0\in X$. Denote by $f_*$ the mapping from  $\pi_1(X,x_0)$ to the fundamental group $\pi_1(C_n ({\mathbb C}) \diagup {\mathcal S}_n, f(x_0))\cong \mathcal{B}_n$ induced by $f$.
Two smooth mappings $f_1$ and $f_2$ from $X$ to $C_n ({\mathbb C}) \diagup {\mathcal S}_n$
define special $(0,n+1)$-bundles that are isotopic through special $(0,n+1)$-bundles
if and only if the quasipolynomials $f_1$ and $f_2$ are free isotopic.
They define bundles that are isomorphic (equivalently, they are isotopic through $(0,n)$-bundles with a section) if and only if for a set of generators $e_j$ of the fundamental group $\pi_1(X,x_0)$, for a braid $w\in \mathcal{B}_n$ and integer numbers $k_j$ the equalities
$(f_1)_*(e_j)=w^{-1}(f_2)_*(e_j)w\Delta_n^{2k_j}$ hold.
Indeed, the bundles are isomorphic iff their monodromy homomorphisms are conjugate.
The monodromy mapping classes of the bundles are elements of $\mathfrak{M}(\mathbb{P}^1; \{\infty\},f(x_0))$ and
the braid group on $n$ strands modulo its center $\mathcal{B}_n\diagup \mathcal{Z}_n$
is isomorphic to the mapping class group $\mathfrak{M}(\mathbb{P}^1;\{\infty\},f(x_0))$.

Each smooth $(0,n)$-bundle $\mathfrak{F}= (\mathcal{X},\mathcal{P},\mathring{\mathbold{E}}\cup \mathbold{s},X)$ with a section
over an oriented smooth finite open surface $X$ is isomorphic (equivalently, isotopic through $(0,n)$-bundles with a section)
to a smooth special $(0,n+1)$-bundle. Indeed,
put $\mathring{E}=\mathring{\mathbold{E}}\cap\mathcal{P}^{-1}(x_0)$, $s=\mathbold{s}\cap\mathcal{P}^{-1}(x_0)$ for the base point $x_0\in X$.
A homeomorphism $\varphi: \mathcal{P}^{-1}(x_0)\to \{x_0\}\times\mathbb{P}^1$ with $\varphi(s)=(x_0,\infty)$ and $\varphi(\mathring{E})= \{x_0\}\times \mathring {E}'$ for a set
$\mathring{E}'\subset C_n(\mathbb{C})\diagup \mathcal{C}_n$ induces
an isomorphism ${\rm Is}:\mathfrak{M}(\mathcal{P}^{-1}(x_0); s, \mathring{E})  \to\mathfrak{M}(\mathbb{P}^1;\{\infty\},\mathring {E}')$. The bundle $\mathfrak{F}$ corresponds to a conjugacy class of homomorphisms $\pi_1(X,x_0)\to \mathfrak{M}(\mathcal{P}^{-1}(x_0); s, \mathring{E}) $. The isomorphism between mapping class groups gives us a conjugacy class of homomorphisms $\pi_1(X,x_0)\to \mathfrak{M}(\mathbb{P}^1;\{\infty\},\mathring {E}')$. There is a special $(0,n+1)$-bundle that corresponds to the latter conjugacy class of homomorphisms.
This bundle is isomorphic to $\mathfrak{F}$.

\begin{lemm}\label{lemEl*}
Each holomorphic $(0,n)$-bundle $\mathfrak{F}\stackrel{def}=(\mathcal{X},\mathcal{P},\mathring{\mathbold{E}}
\cup\mathbold{s},X)$ with a holomorphic section over a connected finite open Riemann surface $X$ is holomorphically isomorphic
to a special holomorphic $(0,n+1)$-bundle.
\end{lemm}

\noindent {\bf Proof}.
Choose $D$, $\tilde{D}_j,\, j=0,\ldots,$ and $\tilde U$ as in the proof of Theorem \ref{thmEl1} in Section
\ref{sec:9.2} (see also Section
\ref{sec:2.0}). Lift the bundle $\mathfrak{F}$ to a bundle on the universal covering $\tilde X$ and restrict the lift to a bundle $\tilde{\mathfrak{F}}= (\tilde{\mathcal{P}}^{-1}(\tilde{U}),\tilde{\mathcal{P}},\widetilde{\mathring{\mathbold{E}}}
\cup\widetilde{\mathbold{s}},\tilde{U})$  on $\tilde U$. The lift $\widetilde{\mathring{\mathbold{E}} }$ of ${\mathring{\mathbold{E}} }$
to $\tilde U$ is the union of $n$ connected components $\widetilde{\mathring{\mathbold{E}} }_j,\,j=1,\ldots,n,$ each intersecting each fiber along a single point.

Consider the holomorphic $(0,0)$-bundle $(\tilde{\mathcal{P}}^{-1}(\tilde{U}),\tilde{\mathcal{P}},\tilde{U})$
that is obtained from $\tilde{\mathfrak{F}}$
by forgetting the
distinguished points and the section. All fibers of the $(0,0)$-bundle $(\tilde{\mathcal{P}}^{-1}(\tilde{U}),\tilde{\mathcal{P}},\tilde{U})$ are conformally equivalent compact Riemann surfaces.
Hence, by a theorem of Fischer and Grauert \cite{FiGr} the bundle is locally holomorphically
trivial. This means that for each point $\tilde{x}\in \tilde{U}$ there exists a simply connected open subset $U_{\tilde{x}} \subset \tilde{U}$
containing $\tilde{x}$, and a biholomorphic map $\varphi_{U_{\tilde{x}}}:(\tilde{\mathcal{P}})^{-1}(U_{\tilde{x}})\to U_{\tilde{x}}\times
\mathbb{P}^1$ such that the diagram
\[
\xymatrix{
{\tilde{\mathcal{P}}^{-1}(U_{\tilde{x}})}      \ar[r]^{\varphi_{U_{\tilde{x}}}     } \ar[d]_{\tilde{\mathcal{P}}         } &
{U_{\tilde{x}} \times \mathbb{P}^1}   \ar[dl]^{{\rm  pr}_1} \\
  U_{\tilde{x}}
}
\]
commutes. For each $\tilde{x}'\in U_{\tilde{x}}$ and $j=1,\ldots,n,$ the mapping
$\varphi_{U_{\tilde{x}}}$
takes the point
$\widetilde{\mathring{\mathbold{E}}}_j \cap (\tilde{\mathcal{P}})^{-1}(\tilde{x}')$
to a point denoted by $(\tilde{x}', \zeta_j(\tilde{x}'))$, and takes the point of the section $\mathbold{s} \cap (\tilde{\mathcal{P}})^{-1}(\tilde{x}')$ to a point denoted by $(\tilde{x}',\zeta_{n+1}(\tilde{x}'))$. The $\zeta_j,\, j=1,\ldots,n+1,$ are holomorphic.
Composing if necessary $\varphi_{U_{\tilde{x}}}$ with a holomorphic mapping from $U_{\tilde{x}} \times \mathbb{P}^1$ onto itself which preserves each fiber, we may assume that
$\varphi_{U_{\tilde{x}}}$ maps no point in $\tilde{\mathcal{P}}^{-1}(\tilde{x})\cap (\widetilde{\mathring{\mathbold{E}}} \cup \mathbold{s})$ to $(\tilde{x},\infty)$.

Maybe, after shrinking the $U_{\tilde x}$, the mapping
$$
U_{\tilde x}\times \mathbb{P}^1\ni (\tilde{x}',\zeta)\to w_{\tilde x}(\tilde{x}',\zeta)=\Big(\tilde{x}', -1+2\frac{(\zeta_2(\tilde{x}')-\zeta)( \zeta_{n+1}(\tilde{x}')-\zeta_1(\tilde{x}'))          }{(\zeta_2(\tilde{x}')- \zeta_{1}(\tilde{x}')) (\zeta_{n+1}(\tilde{x}')-\zeta)}\Big)
$$
provides a biholomorphic map $ U_{\tilde x}\times \mathbb{P}^1\toitself$, that takes for each $\tilde{x}'\in U_{\tilde{x}}$ the point $\zeta_1(\tilde{x}')$ to $1$, it maps $\zeta_2(\tilde{x}')$ to $-1$, and $\zeta_{n+1}(\tilde{x}')$ to $\infty$.

Denote by $\psi_{\tilde x}$ the mapping $w_{\tilde x}\circ \varphi_{U_{\tilde{x}}}$
from  $\tilde{\mathcal{ P}}^{-1}(U_{\tilde{x}})$ onto $U_{\tilde{x}}\times \mathbb{P}^1$.
If for two points $\tilde{x},\tilde{y}\in \tilde{U}$ the intersection $U_{\tilde{x}}\cap U_{\tilde{y}}$ is not empty and connected, then the mappings $\psi_{\tilde{ x}}$ and $\psi_{\tilde{ y}}$ coincide on $\tilde{\mathcal{ P}}_1^{-1}(U_{\tilde{x}}\cap U_{\tilde{y}})$, since in each fiber there is an ordered triple of points that is mapped by each of the mappings to the ordered triple $(-1,1,\infty)$.
We get a holomorphic isomorphism $\tilde{\psi}$ from the total space $\tilde{\mathcal{U}}$ of the bundle $\tilde{\mathfrak{F}}$
onto $\tilde{U}\times \mathbb{P}^1$
that maps the first component of $\widetilde{\mathring{\mathbold{E}}}$ to $\tilde{U}\times \{+1\}$, the second to $\tilde{U}\times \{-1\}$, and the section to  $\tilde{U}\times \{\infty\}$. The isomorphism takes the set $\widetilde{\mathring{\mathbold{E}}}$ to a set $\widetilde{\mathring{\mathbold{E'}}}\subset \tilde{U}\times \mathbb{P}^1$, and the section $\widetilde{\mathbold{s}}$ to the set
$\widetilde{\mathbold{s}}^{\infty}=\tilde{U}\times \{\infty\}$.
The identity mapping from $\tilde U$ onto itself together with the mapping $\tilde{\psi}$ define
a holomorphic isomorphism from the bundle $\tilde{\mathfrak{F}}$ onto the bundle $(\tilde{U}\times \mathbb{P}^1,\,{\rm pr}_1,\,\widetilde{\mathring{\mathbold{E'}}}\cup \widetilde{\mathbold{s}}^{\infty}\,, \tilde{U})$.

Recall that the original bundle $\mathfrak{F}$ is obtained from the bundle $\tilde{\mathfrak{F}}$ by gluing for each $j$ the restriction $\tilde{\mathfrak{F}}|\tilde{D}_j$ to the restriction $\tilde{\mathfrak{F}}|\tilde{D}_0$ using a holomorphic isomorphism
$\varphi_j$ from $\tilde{\mathcal{P}}^{-1}(\tilde{D}_0)$ onto $\tilde{\mathcal{P}}^{-1}(\tilde{D}_j)$. A bundle that is holomorphically isomorphic to $\mathfrak{F}$ is obtained as follows. Take the bundle $(\tilde{U}\times \mathbb{P}^1,\,{\rm pr}_1,\,\widetilde{\mathring{\mathbold{E'}}}\cup \widetilde{\mathbold{s}}^{\infty},\, \tilde{U})$ and glue 
for each $j=1,\ldots,n,$ its restrictions to $\tilde{D}_j$ and  $\tilde{D}_0$
together
using the biholomorphic mapping $\tilde{\psi}_{0,j}\stackrel{def}=
\tilde{\psi}\circ \varphi_j\circ (\tilde{\psi})^{-1}\mid \tilde{D}_0 \times \mathbb{P}^1$
between the subsets $\tilde{D}_0 \times \mathbb{P}^1$ and $\tilde{D}_1 \times \mathbb{P}^1$ of the total space $\tilde{U}\times \mathbb{P}^1$ of the bundle  $(\tilde{U}\times \mathbb{P}^1,\,{\rm pr}_1,\,\widetilde{\mathring{\mathbold{E'}}}\cup \widetilde{\mathbold{s}}^{\infty},\, \tilde{U})$.

Since $\tilde{\psi}_{0,j}$ maps conformally the fiber of
$\tilde{\psi}(\tilde{\mathfrak{F}})$ over $(\tilde{x}_0,\zeta) \in  \tilde{D}_0 \times \mathbb{P}^1$ onto the fiber
over $(\tilde{x}_j,\zeta) \in  \tilde{D}_j \times \mathbb{P}^1$, the following equality holds
\begin{align}\label{eqEl50a}
\tilde{\psi}_{0,j}(\tilde{x}_0,\zeta)= \big(\tilde{x}_j,\tilde{a}_j(\tilde{x}_0)+\tilde{\alpha}_j(\tilde{x}_0) \zeta\big)\,, \; (\tilde{x}_0,\zeta) \in \tilde{D}_0\times\mathbb{C}\,,
\end{align}
for holomorphic functions $\tilde{a}_j$ and $\tilde{\alpha}_j$ with $\tilde{\alpha}_j$ nowhere vanishing on $\tilde{D}_0$.

We want to find a biholomorphic fiber preserving holomorphic mapping
$v:\tilde{U}\times \mathbb{P}^1\toitself$ such that for each $j$ the mappings $v\circ
\tilde{\psi}_{0,j}\circ v^{-1}$ have the form
\begin{align}\label{eqEl50c}
v\circ
\tilde{\psi}_{0,j}\circ v^{-1}(\tilde{x}_0,\zeta)=(\tilde{x}_j,\zeta)\,,\;
(\tilde{x}_0,\zeta)\in \tilde{D}_0\times \mathbb{P}^1\,,
\end{align}
where  $\tilde{x}_j\in  \tilde{D}_j$ is the point for which ${\sf P}(\tilde{x}_0)={\sf P}(\tilde{x}_j)$ for the projection  ${\sf P}:\tilde{X}\to X$.
Then the bundle $(v\circ \tilde{\psi})(\tilde{\mathfrak{F}})$ provides a special holomorphic $(0,n+1)$-bundle that is holomorphically isomorphic to $\mathfrak{F}$ by gluing the fibers over $\tilde{x}_j\in  \tilde{D}_j$ to the fibers over $\tilde{x}_0\in  \tilde{D}_0$ using the mapping $(\tilde{x}_0,\zeta)\to (\tilde{x}_j,\zeta)$.

For $ (\tilde{x},\zeta)\in \tilde{U}\times \mathbb{C}$ we may write
\begin{align}\label{eqEl50b}
v^{-1} (\tilde{x},\zeta)=\big(\tilde{x}, a(\tilde{x})+\alpha(\tilde{x})\,\zeta\big)\,,\;  (\tilde{x},\zeta)\in \tilde{U}\times \mathbb{C}\,.
\end{align}
Condition \eqref{eqEl50c} is by \eqref{eqEl50a} and \eqref{eqEl50b} equivalent to the following requirements on $a$ and $\alpha$:
\begin{align}\label{eqEl50'}
\;{\alpha}(\tilde{x}_j) & =\tilde{\alpha}_j(\tilde{x}_0)\,\cdot \, {\alpha}(\tilde{x}_0),\,\,\;\;\;\;\;\;\;\quad\quad j =1,\ldots,n\,,
\end{align}
\begin{align}\label{eqEl50}
{a}(\tilde{x}_j) & =\tilde{a}_j(\tilde{x}_0)+\tilde{\alpha}_j(\tilde{x}_0)\, {a}(\tilde{x}_0)\,,\;\; \;\;  j =1,\ldots,n\,.
\end{align}
The equation \eqref{eqEl50'} leads to
a Second Cousin Problem for ${\alpha_j}$.
If equation \eqref{eqEl50'} is granted, equation \eqref{eqEl50} can be rewritten as
\index{Cousin Problem ! Second}
\begin{equation}\label{eqEl50''}
\frac{{a}(\tilde{x}_j)}{{\alpha}(\tilde{x}_j)}=\frac{{a}(\tilde{x}_0)}{{\alpha}(\tilde{x}_0)}+\frac{\tilde{a}_j(\tilde{x}_0)}{{\alpha}(\tilde{x}_j)}
\,.
\end{equation}
For convenience of the reader we provide the reduction of equation \eqref{eqEl50'} to the Second Cousin Problem.
Let $D$, $V_j$ and ${\sf s}_j$ be the same objects as in Section \ref{sec:9.2}.
We consider an open cover of $X$ as follows. Put $U_0=D$.
$U_0$ is an open topological disc in $X$ that contains the base point $x_0$. Cover each $V_j$ by two simply connected open sets $U_{j^+}$ and $U_{j^-}$
with connected and simply connected intersection, so that the  $U_{j^{\pm}}$ are disjoint from the $U_{k^{\pm}}$ for $j\neq k$, and
for each $U_{j^{\pm}}$ its intersection with $U_0$ is connected and simply connected. We may
also assume that the intersection of at least three of the sets is empty, the intersections of each ${\sf s}_j$ with $U_{j^+}$ and $U_{j^-}$ are connected, and each ${\sf s}_j$ is disjoint from $U_{k^+} \cup U_{k^-}$  for $k\neq j$.
Label the  $U_{j^{\pm}}$ so that walking on ${\sf s}_j\setminus U_0$ (which is contained in $V_j$) in the direction of the orientation of ${\sf s}_j$
we first meet $U_{j^-}$.

\begin{figure}[h]
\begin{center}
\includegraphics[width=60mm]{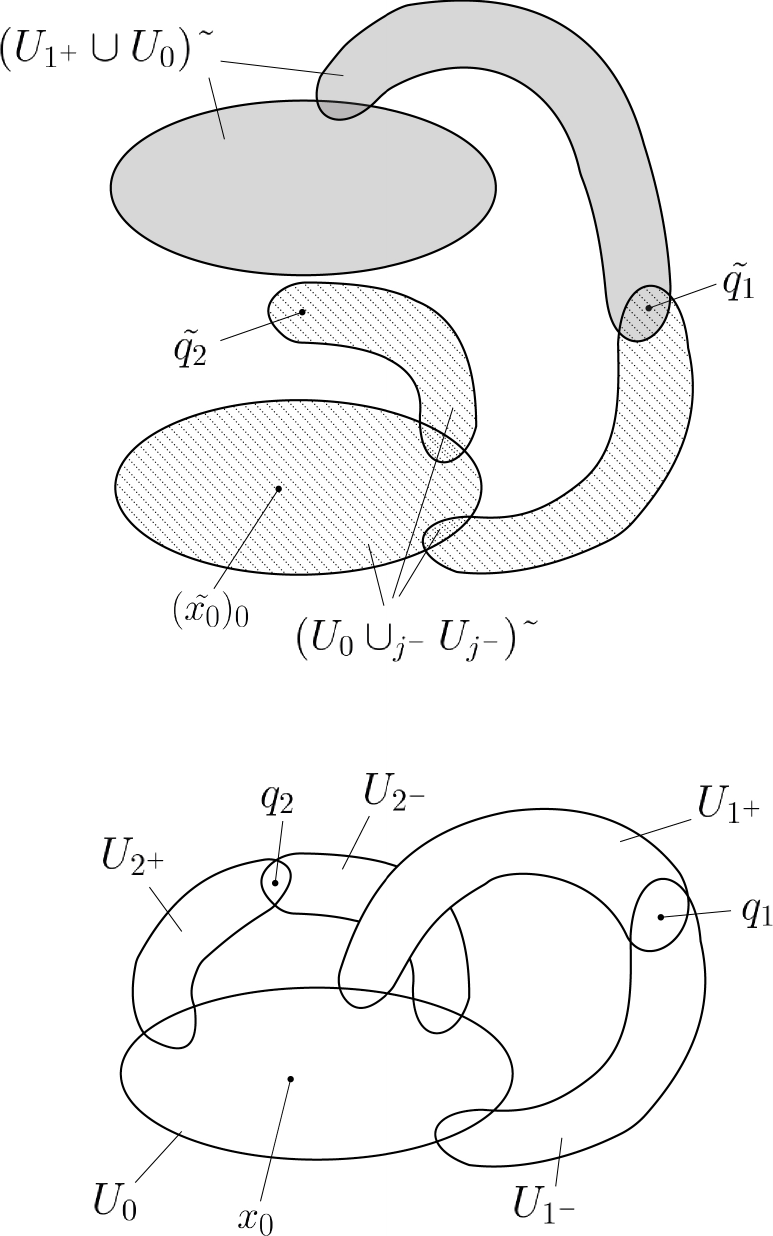}
\end{center}
\caption{The open cover of $X$ and its lift to $\tilde X$}\label{figEl4}
\end{figure}

We consider the following subsets of $\tilde X$. Let
$(U_0\cup_{j^-}\, U_{j^-}){\;\widetilde{}}\;$  be a lift of $U_0\cup_{j^-}\, U_{j^-}$ to
$\tilde X$. Let $(\tilde{x_0})_0$ be the point in the lifted set that projects to $x_0$. For
each $j\geq 1$ we consider a point ${q}_j\in U_{j^-} \cap U_{j^+}$ and the point
$\tilde{q}_j$ in
$(U_0\bigcup \cup_{j^-} U_{j^-}){\;\widetilde{}}\;$ that projects to $q_j$.
For each $j$ we denote by
$(U_{j^+}\cup U_0){\;\widetilde{}}\;$  the lift of  $U_{j^+}\cup U_0$ to $\tilde X$ that
contains $\tilde{q}_j$. Let $\widetilde{U_0^{j^+}}$ be the subset of
$(U_{j^+}\cup U_0){\;\widetilde{}}\;$ that projects to $U_0$.

The sets $U_0$, $U_{j^{\pm}}$ cover $X$.
The intersection of three distinct sets of the cover is empty. On the intersection of
pairs of covering sets we define a non-vanishing holomorphic function as follows.
Put
\begin{align}\label{eqEl26}
\alpha_{0,j^-}\;\;=(\,\alpha_{j^-,0}\;)^{-1}= 1 \quad \quad \quad \;&\mbox{on}\; U_0 \;\;
\cap U_{j^-} \quad  \,j =1,\ldots, 2g+m\,,\nonumber\\
\alpha_{j^-,j^+}=(\alpha_{j^+,j^-})^{-1}= 1 \quad \quad \quad \;&\mbox{on}\; U_{j^-} \cap
U_{j^+} \quad  \,j =1,\ldots, 2g+m\,,\nonumber\\
\alpha_{j^+,0}\;\; = (\,\alpha_{0,j^+}\;)^{-1}= \tilde{\alpha}_{j,0} \quad 
\quad \,&\mbox{on}\;U_{j^+}\cap  U_0
\;\; \quad \,j =1,\ldots, 2g+m\,,
\end{align}
where for $U_{j^+} \cap U_0,$  the function $\tilde{\alpha}_{j,0}$ is defined as
$\tilde{\alpha}_{j,0}(x)= \tilde{\alpha}_j(\tilde{x}_0)$ for the lift $\tilde{x}_0$ of the point $x\in U_0$ to  $\widetilde{U_0^{j^+}}$.

Since there are no triple intersections, the equations \eqref{eqEl26} define a Cousin II distribution.
A second Cousin distribution on a complex manifold defines a holomorphic line bundle. The Second Cousin Problem has a solution if and only if the line bundle is holomorphically trivial. Since on an open Riemann surface each holomorphic line bundle is holomorphically trivial the Cousin Problem has a solution (see e.g. \cite{Fo}).

The solution of the Cousin Problem for the Cousin distribution
\eqref{eqEl26}
provides non-vanishing holomorphic functions $\alpha_0$ on $U_0$ and
$\alpha_{j^\pm}$ on $U_{j^\pm}$ such that
\begin{align}\label{eqEl27}
\alpha_0 \;\;\;\;\quad& =\alpha_{j^-} \quad\quad \;\;\;\; \mbox{on}\; U_0\;\,\cap
U_{j^-},\nonumber\\
\alpha_{j^-}\;\;\quad& = \alpha_{j^+} \quad\quad \;\;\;\; \mbox{on}\; U_{j^-}\cap
U_{j^+},\,\nonumber\\
\alpha_{j^+} \alpha_{0}^{-1}& =\tilde{\alpha}_{j,0} \quad\quad
\;\;\;\; \mbox{on}\; U_{j^+}\cap U_0\;\;.
\end{align}
Put
\begin{equation}
{\alpha}(\tilde{x})=\begin{cases}
\alpha_0({x})& \tilde{x}\in\tilde{U}_0\\
\alpha_{j^-}({x})& \tilde{x}\in\tilde{U}_{j^-}\\
\alpha_{j^+}({x})& \tilde{x}\in\tilde{U}_{j^+}\\
\tilde{\alpha}_{j,0}(x)\cdot \alpha_0({x})& \tilde{x}\in \widetilde{U_0^{j^+}}\,.
\end{cases}
\end{equation}
Here $x={\sf P}(\tilde{x})$ for the projection ${\sf P}:\tilde{X}\to X$.
The function ${\alpha}$ is a well-defined holomorphic function on $\tilde U$ that satisfies \eqref{eqEl50'}.

The reduction of \eqref{eqEl50''} to a First Cousin Problem for the
$\frac{{a}(\tilde{x}_j)}{{\alpha}(\tilde{x}_j)}$, $j=0,\ldots 2g+m-1,$
is similar. The First Cousin Problem is solvable since open Riemann surfaces are Stein manifolds (see e.g. \cite{Fo} and \cite{H1}, and also Appendix \ref{ChapterA}).
\hfill $\Box$
\index{Cousin Problem ! First}
\medskip
\section[$(0,4)$-bundles over tori with a hole. Proof of Theorem 8.4]  {
$(0,4)$-bundles over tori with a hole. Proof of Theorem \ref{thmGrom2}.}\label{sec:9.4a}

We will consider now deformations of a smooth $(0,3)$-bundle with a section over a connected finite open Riemann surface $X$  to a holomorphic bundle.
Statement $(1)$ of Theorem \ref{thmEl.0} below  is the analog of Theorem \ref{thmGrom1} for $(0,3)$-bundles with a section.
\begin{defn}\label{defnEl.3'}
We will say that a smooth $({\sf g,m})$-bundle $\mathfrak{F}$ over a connected smooth finite open oriented surface $X$ is isotopic to a holomorphic bundle for the conformal structure $\omega:X\to \omega(X)$, if the pushed forward bundle $\mathfrak{F}_{\omega}$ (see equation \eqref{eqEl23}) is isotopic to a holomorphic bundle on $\omega(X)$. If the bundle is isotopic to a holomorphic bundle for each conformal structure with only thick ends on $X$, then $\mathfrak{F}$ is said to have the Gromov-Oka property.

A $(0,n)$-bundle on $X$ with a section, in particular, a special $(0,n+1)$-bundle on $X$, is said to have the Gromov-Oka property, if and only if for each conformal structure on $X$ with only thick ends the bundle is isotopic through
$(0,n)$-bundles with a section to a holomorphic bundle.
\end{defn}
\index{Gromov-Oka property} \index{$\mathfrak{F}_{\omega}$}

In the following theorem we start with a smooth special $(0,4)$-bundle, since each smooth $(0,3)$-bundle with a section is isotopic to a special $(0,4)$-bundle.
Recall that each connected finite open Riemann surface $\mathcal{X}$ is conformally equivalent to a domain in a connected closed Riemann surface $\mathcal{X}^c$.
Notice that a special $(0,4)$-bundle is reducible if and only if all monodromies are powers of a single conjugate of $\sigma_1$.

\begin{thm}\label{thmEl.0}
$(1)$ Let $X$ be a connected smooth oriented surface of genus one with a hole with base point $x_0$, and with a chosen set
$\mathcal{E}=\{e_1,e_2\}$ of generators of $\pi_1(X,x_0)$. Define the set
$\mathcal{E}_0=
\{e_1,e_2, e_1 e_2^{-1},  e_1 e_2^{-2}, e_1e_2 e_1^{-1} e_2^{-1}\}$ as in Theorem  $\ref{thmGrom1}$. Consider a smooth special $(0,4)$-bundle $\mathfrak{F}$ over $X$. Suppose for each $e\in \mathcal{E}_0$ the restriction of the bundle $\mathfrak{F}$ to an annulus
in $X$ representing $\hat e$ has the Gromov-Oka property.
Then the bundle over $X$ has the Gromov-Oka property.\\
$(2)$ If a bundle $\mathfrak{F}$ as in $(1)$ is irreducible, then it is isotopic to an isotrivial bundle, and, hence,
for any conformal structure $\omega$ on $X$
the bundle $\mathfrak{F}_{\omega}$ is isotopic to a bundle that extends to a holomorphic bundle on $\omega(X)^c$. In particular, the bundle is isotopic to a holomorphic bundle for any conformal structure on $X$ (including conformal structures of first kind).\\
$(3)$ Let $\mathfrak{F}$ be any smooth reducible special $(0,4)$-bundle over $X$ (without any further requirement). Then each irreducible bundle component of  $\,\mathfrak{F}$ is isotopic to an
isotrivial bundle. There exists a Dehn twist in the fiber over the base point
such that the bundle $\mathfrak{F}$ can be recovered from the irreducible bundle components up to composing each monodromy with a power of this Dehn twist. The bundle $\mathfrak{F}$ itself is isotopic to a holomorphic bundle for each conformal structure of second kind on $X$.\\
$(4)$ Any reducible holomorphic $(0,3)$-bundle with a holomorphic section over a ${\rm punctured}$ Riemann surface is holomorphically trivial.
\end{thm}

Recall that a bundle is called isotrivial if its lift to a finite covering of $X$ is the trivial bundle.
We will prove now Theorems \ref{thmGrom1} and \ref{thmEl.0} using Theorem \ref{thmGrom2}.
\medskip

\noindent {\bf Proof of Theorem \ref{thmEl.0}.}
The special $(0,4)$-bundle $\mathfrak{F}$ defines a smooth quasipolynomial $f$ on $X$. Under the conditions of Statement (1) for each $e\in \mathcal{E}_0$ the restriction of the bundle to an annulus $A_{\hat{e}}$ in $X$ representing $\hat e$ has the Gromov-Oka property. By Lemma \ref{lemEl*}
we may assume that for a conformal structure of conformal module bigger than $\frac{\pi}{2}\log(\frac{3+\sqrt{5}}{2})^{-1}$ on $A_{\hat{e}}$ the restriction $\mathfrak{F}\mid A_{\hat{e}}$ is isotopic (through $(0,3)$-bundles with a section) to a special holomorphic $(0,4)$-bundle $\mathfrak{F}_{\hat{e}}$ (not merely to a holomorphic $(0,3)$-bundle with a section). The bundle $\mathfrak{F}_{\hat{e}}$ defines a holomorphic quasipolynomial $f_{\hat{e}}$ on $A_{\hat{e}}$. The monodromies of the isotopic (hence, isomorphic) bundles  $\mathfrak{F}\mid A_{\hat{e}}$ and $\mathfrak{F}_{\hat{e}}$ differ by conjugation (after identification of the mapping class groups on the fiber over the base point by an isomorphism).
Hence, for each $e\in\mathcal{E}$ there exists an integer number $k_e$,
such that the quasipolynomial $f \mid A_{\hat{e}}$ is (free) isotopic to the holomorphic quasipolynomial
$z \to e^{2\pi ik_e z} \,f_{\hat{e}}(z),\, z\in  A_{\hat{e}}$. We proved that the quasipolynomial $f$ satisfies the conditions of Theorem \ref{thmGrom2}.

The monodromy mapping classes of the quasipolynomial $f$ are elements of the braid group $\mathcal{B}_3$, the monodromy mapping classes of the bundle
$\mathfrak{F}$ can be identified with elements of $\mathcal{B}_3\diagup \mathcal{Z}_3$.
In terms of the $(0,4)$-bundle $\mathfrak{F}$ the Theorem \ref{thmGrom2}
implies the following. The isotopy class of the special $(0,4)$-bundle
$\mathfrak{F}$ on $X$ corresponds to the conjugacy class of a homomorphism $\mathbold{\Phi}:\pi_1(X,x_0)\to
\mathbold{\Gamma} \subset \mathcal{B}_3\diagup \mathcal{Z}_3$, where
$\mathbold{\Gamma}$
is generated either by $\sigma_1 \,\sigma_2\diagup \mathcal{Z}_3$, or by
$\Delta_3\diagup\mathcal{Z}_3$, or by $\sigma_1\diagup\mathcal{Z}_3$.

The group $\mathbold{\Gamma}$ is Abelian. Let $\omega$ be any conformal structure on $X$ (maybe, of first kind). The
fundamental group $\pi_1 (\, \omega(X)^{c}, \omega(x_0))$ of the closed torus
$\,\omega(X)^{c}$ (that contains a conformal copy of $\omega(X)$) is the Abelianization of $\pi_1 (\omega(X),\omega(x_0))\,\cong\, \pi_1 ( \, X , x_0)\,$. Hence
$\,\mathbold{\Phi}\,$ defines a homomorphism from $\;\pi_1 (\, \omega(X)^{c}, \omega(x_0))\,$ to
$\mathbold{\Gamma}$ which we also denote by $\mathbold{\Phi}$.

Consider first the case when the generator of $\mathbold{\Gamma}$ is either
$\sigma_1\sigma_2 \diagup \mathcal{Z}_3$, or
$\Delta_3\diagup \mathcal{Z}_3$. By Lemma \ref{lemEl.0} this is exactly the case
when the bundle $\mathfrak{F}$ is irreducible.
The generator of $\mathbold{\Gamma}$ corresponds to an element of the
mapping class group  $\mathfrak{M}(\mathbb{P}^1; \{\infty\},\mathring{ E}_3)$,
that is represented by a periodic mapping $\zeta\to \theta \cdot \zeta,\,
\zeta \in \mathbb{P}^1$ (see \cite{Be1}). In the first case $\theta = e^{\frac{2\pi i}3}$ and $\mathring{E}_3$ consists of
$3$ equidistributed points on a circle with center zero,
in the second case  $\theta= e^{\frac{2\pi i}{2}}$ and $\mathring{E}_3$ consists of the origin and two
equidistributed points on a circle with center zero.

Consider the homomorphism from $\mathbold{\Gamma}$ to the group of complex
linear transformations of ${\mathbb P}^1$, which assigns to the
generator $\mathbold{b} \in \mathbold{\Gamma}$ the complex linear transformation $\zeta \to \theta
\cdot \zeta$, $\zeta \in {\mathbb P}^1$.
Composing $\mathbold{\Phi}$ with
this homomorphism we obtain a homomorphism $\Xi$ from $\pi_1 (\, \omega(X)^{c}, \omega(x_0))$
to the group of complex linear transformations of ${\mathbb
P}^1$.

Represent the closed torus $\, \omega(X)^{c}$ as quotient $ \,\omega(X)^{c}  = {\mathbb C} \diagup \Lambda$ for a lattice $\Lambda
=\{\lambda_1 \, k_1 + \lambda_2 \, k_2 : k_1 , k_2 \in {\mathbb
Z}\}$ with
complex numbers $\lambda_1 , \lambda_2
\in {\mathbb C}$ that are linearly independent over $\mathbb{R}$. Denote by $p_1$ the covering map $p_1 : {\mathbb
C} \to {\mathbb C} \diagup \Lambda$. Each element $\lambda\in \Lambda$ defines a covering transformation $z\to z+\lambda,\, z\in \mathbb{C}$. We identify $\Lambda$ with the group of covering transformations.
Recall that the fundamental group
$ \pi_1 ( \, \omega(X)^{c},
\omega(x_0))$ is isomorphic to the group of covering transformations of  $ \, \omega(X)^{c} \cong {\mathbb C}
\diagup \Lambda$.
\index{lattice} \index{$\Lambda$}

The group $\Lambda$ acts on ${\mathbb C} \times {\mathbb P}^1$ as
follows. Associate to $\lambda \in \Lambda$ the mapping
\begin{equation}
\label{eqEl81} (z,\zeta) \overset{\lambda}{\longrightarrow} (z+\lambda, \Xi (\lambda) (\zeta)) \, , \quad (z,\zeta) \in {\mathbb C}
\times {\mathbb P}^1 \, .
\end{equation}
The thus defined mapping $\lambda$ takes the set $\mathbb{C}\times\mathring{E}_3$ to itself and fixes the set $\mathbb{C}\times\{\infty\}$.
The action is free and properly discontinuous. Hence the quotient
$({\mathbb C} \times {\mathbb P}^1) \diagup \Lambda$ is a complex
manifold. For each element of the quotient its projection to
${\mathbb C} \diagup \Lambda$ is well-defined. Indeed, the following diagram commutes:
$$
\xymatrix{
(z,\zeta) \ar[d] &\sim&(\lambda + z , \Xi (\lambda) (\zeta)) \ar[d] \\
z &\sim&\lambda + z }
$$
The equivalence relation in the upper line gives the quotient
$({\mathbb C} \times {\mathbb P}^1) \diagup \Lambda$, the relation
in the lower line gives ${\mathbb C} \diagup \Lambda$.

Denote by $ \mathcal{P}$
the projection $\mathcal{P}:\big({\mathbb C} \times {\mathbb P}^1\big)\diagup \Lambda
\to {\mathbb C} \diagup \Lambda$.
We obtain a holomorphic $(0,3)$-bundle
$$
\mathfrak{F}_{\omega}\stackrel{def}=\Bigg(({\mathbb C} \times {\mathbb P}^1)\diagup \Lambda,\;\mathcal{P},\; \Big((\mathbb{C}\times\mathring{E}_3)\diagup\Lambda\Big) \cup\Big((\mathbb{C}\times\{\infty\})\diagup \Lambda\Big),\;{\mathbb C} \diagup \Lambda\Bigg)
$$
with a holomorphic section over the closed torus $\omega(X)^c$. (See also  \eqref{eqEl23} for notation.)
By construction the restriction of this bundle to a representative of each generator
$e_j$ of $\pi_1(\omega(X)^{c},x_0)$ gives the mapping torus corresponding to $\Xi(e_j)$. Hence, the
monodromy mapping classes of the bundle are equal to $\mathbold{\Phi}(e_j)$. We proved that
in the irreducible case of Theorem \ref{thmEl.0} for any conformal structure $\omega$ on $X$ the bundle $\mathfrak{F}_{\omega}$ is smoothly isomorphic (equivalently, isotopic) to a holomorphic
bundle that extends to a holomorphic bundle $(\mathfrak{F}_{\omega})^c$ on the closed torus $\omega(X)^c$.

Since the monodromy mapping classes of the bundle $\mathfrak{F}_{ \omega}$ are periodic, the lift of the bundle $(\mathfrak{F}_{ \omega})^c$ to a finite covering of $X$ is isotopic to an isotrivial bundle.
Statement (2) is proved and also statement (1) in the irreducible case.

To prove Statement (3), we
suppose $\mathfrak{F}$ is an arbitrary reducible smooth special $(0,4)$-bundle.
Let $\mathring{E}_3$ be the set of finite distinguished points in the fiber over the base point.
There is an admissible curve $\gamma\subset \mathbb{P}^1\setminus (\mathring{E}_3\cup \{\infty\})$ that reduces each monodromy mapping class of the bundle. Hence, by (the proof of) Lemma \ref{lemEl.0} each monodromy mapping class of the bundle is a non-trivial
power of a single conjugate of the mapping class $\mathfrak{m}({\sigma_1\diagup \mathcal{Z}_3})$ in
$\mathfrak{M}(\mathbb{P}^1;\infty,\mathring{E}_3)$ that corresponds to  ${\sigma_1\diagup \mathcal{Z}_3}$. We are in the case when $\mathbold{\Gamma}$ is generated by $\sigma_1\diagup \mathcal{Z}_3$.
Label the points of $\mathring{E}_3$ as in the proof of Lemma \ref{lemEl.0} so that
the simple closed curve $\gamma$ separates $\{\zeta_1,\zeta_2\}$ from $\{\zeta_3,\infty\}$.
Let $\mathcal{C}_1$ be the connected component of the complement of $\gamma$ that contains $\{\zeta_1,\zeta_2\}$, and let $\mathcal{C}_2$ be the connected component of the complement of $\gamma$ that contains $\{\zeta_3,\infty\}$.

There are two irreducible components of the
mapping class $\mathfrak{m}({\sigma_1\diagup \mathcal{Z}_3})$. The irreducible component  that corresponds to $\mathcal{C}_1$ is described as follows.
The connected component $\mathcal{C}_1$ is
homeomorphic to $\mathbb{P}^1\setminus \{\infty\}$ with set of distinguished points $\{\zeta_1,\zeta_2\}$. Each self-homeomorphism of $\mathbb{P}^1\setminus \{\infty\}$ with these two distinguished points extends to a self-homeomorphism of $\mathbb{P}^1$ that maps the set $\{\zeta_1,\zeta_2, \infty\}$ of three distinguished points to itself.
The irreducible component of the mapping class $\mathfrak{m}({\sigma_1\diagup
\mathcal{Z}_3})$ corresponding to $\mathcal{C}_1$ is an element of
the mapping class group $\mathfrak{M}(\mathbb{P}^1; \{\infty\}, \{\zeta_1,\zeta_2\})$.
This mapping class group is generated by the element $\mathbold{\sigma}=\sigma\diagup \mathcal{Z}_2$ which
corresponds to the braid $\sigma$ in the braid group $\mathcal{B}_2$ on two strands. The
square of $\mathbold{\sigma}$ is the identity. We put  $\mathbold{\Gamma}_1=
\mathfrak{M}(\mathbb{P}^1; \{\infty\}, \{\zeta_1,\zeta_2\})$.

To describe the irreducible component corresponding to $\mathcal{C}_2$ of the mapping class $\mathfrak{m}(\sigma_1\diagup \mathcal{Z}_3)$, we notice that
$\mathcal{C}_2$ is homeomorphic to $\mathbb{P}^1\setminus \{\zeta_4\}$ with distinguished points $\{\zeta_3,\infty\}$ for a point $\zeta_4\in \mathbb{P}^1, \, \zeta_4\neq \zeta_3,\infty$.
Each self-homeomorphism of $\mathcal{P}^1\setminus \{\zeta_4\}$ with distinguished points
$\{\zeta_3,\infty\}$ extends to a self-homeomorphism of
$\mathbb{P}^1$ with set of distinguished points $\{\zeta_3,\zeta_4,\infty\}$. The respective irreducible component of $\mathfrak{m}({\sigma_1\diagup \mathcal{Z}_3})$ is the identity in the group $\mathfrak{M}(\mathbb{P}^1; \{\zeta_4, \infty\},z_3)$.

All monodromy mapping classes of a bundle that is isomorphic to $\mathfrak{F}$ are contained in $\mathbold{\Gamma}$, hence are powers of $\mathfrak{m}({\sigma_1\diagup \mathcal{Z}_3})$.
This implies that there are two irreducible bundle components, corresponding to $\mathcal{C}_1$ and to $\mathcal{C}_2$, respectively.
They are described as follows.

The irreducible bundle component of the bundle $\mathfrak{F}$, corresponding to $\mathcal{C}_2$, is associated to the trivial homomorphism. Hence, this irreducible bundle component is smoothly isomorphic to the trivial $(0,3)$-bundle over $X$.

The irreducible bundle component of $\,\mathfrak{F}\,$ related to $\mathcal{C}_1$ corresponds to the conjugacy
class of the homomorphism $\;\mathbold{\Phi}_1: \pi_1(X,x_0)\to \mathbold{\Gamma}_1$
induced by the homomorphism $\mathbold{\Phi}$. The same proof as in the case,
when $\mathbold{\Gamma}$ is generated either by
$\sigma_1\sigma_2 \diagup \mathcal{Z}_3$, or by
$\Delta_3\diagup \mathcal{Z}_3$,
shows that for any conformal structure $\omega$ on $X$ there is a holomorphic
$(0,2)$-bundle $(\mathfrak{F}'_1)^{\omega}$
with a holomorphic section
on the closed torus $\omega(X)^{c}$ whose isomorphism class corresponds to
the conjugacy class of the homomorphism $\pi_1(\omega(X)^c,\omega(x_0)) \to
\mathbold{\Gamma}_1$ that is induced by $\mathbold{\Phi}_1$.
The restriction of $(\mathfrak{F}'_1)^{\omega}$ to a punctured torus $\omega({X})'$, $\omega(X)\Subset \omega({X})'$,
is isomorphic to a special holomorphic $(0,3)$-bundle $(\mathfrak{F}_1)^{\omega}$ with set of distinguished points
$\{\zeta_1(x),\zeta_2(x),\infty\}$ in the fiber $\{x\}\times \mathbb{P}^1$ over $x$. For
$x\in \omega(X) \Subset \omega(X)'$ the points $\zeta_1(x),\zeta_2(x)$ are contained in a large closed disc in
$\mathbb{C}$ centered at the origin.

Take a point $\zeta_3\in \mathbb{C}$ outside this closed disc.
Consider the special holomorphic $(0,4)$-bundle $(\mathfrak{F})^{\omega}
\stackrel{def}=\Big(X \times \mathbb{P}^1, {\rm pr}_1, \mathring{\mathbold{E}}_3\cup\mathbold{s}^{\infty}, X\Big)$,
where the set of finite distinguished points $\mathring{\mathbold{E}}_3\cap(\{x\}\times \mathbb{P}^1)$ in the fiber over $x$ equals
$\{x\}\times\{\zeta_1(x) ,\zeta_2(x), \zeta_3\}$ and $\mathbold{s}^{\infty}\cap(\{x\}\times \mathbb{P}^1)=(x,\infty)$.
The monodromy mapping class of the bundle $(\mathfrak{F})^{ \omega}$ along $e_j$ differs from the monodromy mapping class $\mathfrak{F}_{ \omega}$, which is the push-forward  to $\omega(X)$ of the bundle $\mathfrak{F}$, by the $k_j$-th power of a Dehn twist about the curve $\gamma$.
Consider again the punctured torus $\omega(X)'\supset \omega(X)$, and assume
that $\omega(X)'=(\mathbb{C}\setminus \Lambda) \diagup \Lambda$
with
$\Lambda = \{ \lambda_1 \, n_1 + \lambda_2 \, n_2 \, , \ n_1 ,
n_2 \in {\mathbb Z}\}$, and $\omega(x_0)=\frac{\lambda_1+\lambda_2}{2}\diagup \Lambda$, and that the segments $\frac{\lambda_1+\lambda_2}{2}+[0,\lambda_1]$ and $\frac{\lambda_1+\lambda_2}{2}+[0,\lambda_2]$ are lifts to $\mathbb{C}\setminus \Lambda$ of curves representing a pair of generators of $\pi_1 (\omega(X)', \omega(x_0))$.

For all integers $\ell_1$ and $\ell_2$ there is a holomorphic function $F$ on ${\mathbb C}
\backslash
\Lambda$ such that
\begin{equation}
\label{eqEl018} F(x+\lambda_1) = \ell_1 \quad {\rm and} \quad F(x +
\lambda_2) = \ell_2, \; x \in \mathbb{C}\setminus \Lambda\,.
\end{equation}
This follows from the proof of Lemma \ref{lemGrom3}.
A quick way to see this for $\ell_2 = 0$ and any $\ell_1 \in
{\mathbb Z}$ is the following. For each $x\in \mathbb{C}$ there are uniquely determined
real numbers $t_1(x)$ and $t_2(x)$ depending smoothly on $x$ such that $x=\lambda_1 t_1(x)+\lambda_2 t_2(x)$.
Consider a function $\chi_0$ on the real axis which vanishes near zero such that
$\chi_0(t+1)=\chi_0(t)+\ell_1$, $t \in \mathbb{R}$. Consider the function $\chi(x)\stackrel{def}=\chi_0(t_1(x)),\, x\in \mathbb{C}$.
The
1-form $\overline\partial \, \chi$ descends to a smooth 1-form
$\delta$ on the torus ${\mathbb C} \diagup \Lambda$. Let $\mathring{g}$ be a
solution of the equation $\overline\partial \, \mathring{g} = \delta$ on the
punctured torus ${\mathbb C} \backslash \Lambda \diagup \Lambda$.
Let $g$ be the lift of $\mathring{g}$ to ${\mathbb C}
\backslash \Lambda$. Put $F = \chi - {g}$.

Let $F$ be the holomorphic function on
${\mathbb C} \backslash \Lambda$ that satisfies equation
\eqref{eqEl018}
for the integers $\ell_1=-k_1$ and $\ell_2=-k_2$.
Then the monodromy mapping classes along the $e_j$ of the special holomorphic $(0,4)$-bundle $(\mathfrak{F}')^{\omega}$ over $\omega(X)$
with set of distinguished points
$\{\zeta_1(x) ,\zeta_2(x), \zeta_3 \,e^{2\pi i F(x)},\infty\}$ in each fiber $\{x\}\times \mathbb{P}^1$, $x\in \omega(X)$,
are the same as the monodromy mapping classes of the bundle $\mathfrak{F}_{ \omega}$. Statement (3) and, hence, statement (1) in the reducible case are proved.

It remains to prove Statement (4), namely the fact that any reducible holomorphic $(0,3)$-bundle $\mathfrak{F}$
with a holomorphic section over a {\it punctured} Riemann surface $X$ is holomorphically isomorphic to the trivial bundle. The respective statement is known in more general situations,
but the proof of the present statement is a simple reduction to Picard's Theorem and we include it.

By Lemma \ref{lemEl*} we may assume that the bundle is equal to a special $(0,4)$-bundle
$\mathfrak{F}=(X\times \mathbb{P}^1,{\rm pr}_1, \mathring{\mathbold{E}}\cup \mathbold{s}^{\infty},X)$.
There is a finite unramified covering $\hat X$ of the punctured Riemann surface $X$
so that for the lift $\hat {\mathfrak{F}}= \big(\hat{X}\times \mathbb{P}^1,{\rm pr}_1,\widehat{\mathring{\mathbold{E}}}\cup \hat{\mathbold{s}}^{\infty},\hat{X}\big)$ of the bundle $\mathfrak{F}$ to $\hat X$ the set
$\widehat{\mathring{\mathbold{E}}}\cup \hat{\mathbold{s}}^{\infty}$ is the union of four disjoint complex curves each intersecting each fiber along a single point. After a holomorphic isomorphism (see the proof of Lemma \ref{lemEl*}) we may assume that the connected components of $\widehat{\mathring{\mathbold{E}}}$ are $\hat{X}\times\{-1\}$, $\hat{X}\times\{1\}$, and a component that intersects each fiber $\hat{x}\times \mathbb{P}^1$ along a point $(\hat{x},\hat {g}(\hat{x}))$ for a holomorphic function $g$ on $\hat X$ that omits the values $-1,\,1$, and $\infty$.

Since the bundle ${\mathfrak{F}}$ is reducible, the bundle $\hat{\mathfrak{F}}$ is reducible.
Hence, $\hat{g}:\hat{X}\to \mathbb{C}\setminus \{-1,1\}$ is free homotopic to a mapping whose image is contained in a punctured disc
around one of the points $-1,\, 1,$ or $\infty$.
We may assume, after applying to each fiber the same conformal self-mapping of $\mathbb{P}^1$, that this point is $-1$.

By Picard's Theorem the mapping $\hat g$ extends to a meromorphic mapping $\hat{g}^c$ from the closure $\hat{X}^c$ of $\hat X$ to $\mathbb{P}^1$. \index{Picard Theorem}
Then the meromorphic extension $\hat{g}^c$ omits the value $1$. Indeed, if $\hat{g}^c$ was equal to $1$ at some puncture of $\hat X$, then $\hat g$ would map a simple closed curve on $\hat X$ that surrounds the puncture positively to a loop in $\mathbb{C}\setminus \{-1,1\}$ that surrounds $1$ with positive winding number. This
contradicts the fact that $\hat g$ is homotopic to a mapping into a disc punctured at $-1$ and contained in $\mathbb{C}\setminus \{-1,1\}$. Hence, $\hat{g}^c$ is a meromorphic function on a compact Riemann surface that omits a value, and, hence $\hat g$ is constant.

This means that the bundle $\hat{\mathfrak{F}}$ over $\hat X$ is holomorphically isomorphic to the trivial bundle. Hence, the monodromy mappings of the original bundle $\mathfrak{F}$ are periodic. By Lemma \ref{lemEl.0}
this is possible for a reducible bundle only if the monodromies are equal to the identity. Repeat the same reasoning as above for the bundle $\mathfrak{F}$ instead of $\hat{\mathfrak{F}}$,
we see that the bundle $\mathfrak{F}$ is holomorphically trivial.
Theorem \ref{thmEl.0} is proved. \hfill $\Box$

\medskip

\noindent {\bf Proof of Theorem \ref{thmGrom1}.}
By Theorem \ref{thmGrom2} the isotopy class of the separable quasipolynomial $f$ corresponds to the
conjugacy class of a homomorphism $\Phi: \pi_1(X,x_0) \to \Gamma \subset \mathcal{B}_3$,
where $\Gamma$ is generated either by $\sigma_1\sigma_2$, or by
$\Delta_3$, or by $\sigma_1$ and $\Delta_3^2$. Consider the special $(0,4)$-bundle that is
associated to the quasipolynomial.
The isotopy class of the bundle over $X$
(with respect to isotopies through $(0,3)$-bundles with a section)
corresponds to the conjugacy class of the homomorphism $\mathbold{\Phi}: \pi_1(X,x_0) \to
\mathbold{\Gamma}=\Gamma\diagup\mathcal{Z}_3 $ that is obtained from $\Phi$ by taking the
quotient with respect to $\mathcal{Z}_3$.

By Statement (1) of Theorem \ref{thmEl.0} the conditions of Theorem \ref{thmGrom1} imply that for any a priori given conformal structure $\omega$ of second kind on $X$ the isotopy class of the bundle contains a special holomorphic $(0,4)$-bundle.
The special holomorphic
$(0,4)$-bundle determines a holomorphic quasipolynomial $f^{\omega}$ for this conformal structure by
assigning to each $x\in \omega(X)$ the triple of finite distinguished points in the fiber
$\{x\}\times \mathbb{P}^1$ and associating to it the monic polynomial with this set of zeros.
We may assume by conjugating the monodromy mapping of the special $(0,4)$-bundle, that
the monodromies of the quasipolynomial $f^{\omega}$ differ from the monodromies of the push forward $f_{\omega}$ of the original quasipolynomial by powers of $\Delta_3^2$.
Using the function $F$ that satisfies
\eqref{eqEl018} for suitable integers $\ell_1$ and $\ell_2$ we obtain a holomorphic quasipolynomial $f^{\omega}(x) e^{2\pi i F(x)},\, x \in \omega(X),$    that is isotopic to $f_{\omega}$.
Theorem \ref{thmGrom1} is proved. \hfill $\Box$

\section[Smooth fiber bundles and smooth families of complex manifolds.] {Smooth elliptic fiber bundles and differentiable families of complex manifolds.}
\label{sec:9.4}
\index{fiber bundle !  elliptic}

Recall that Kodaira (see \cite{Kodai}) defines a complex analytic family (also called holomorphic family) $\mathcal{M}_t, \,t
\in \mathcal{B},$ of compact complex manifolds over the base $\mathcal{B}$
as a triple $(\mathcal{M}, \mathcal{P}, \mathcal{B})$, where $\mathcal{M}$ and $\mathcal{B}$
are complex  manifolds, and $\mathcal{P}$ is a proper holomorphic submersion with
$\mathcal{P}^{-1}(t) = \mathcal{M}_t$. Thus, each holomorphic genus $\sf{g}$ fiber bundle
over a Riemann surface is a complex analytic family of compact Riemann surfaces of genus
$\sf{g}$.

We will consider holomorphic $(\sf{g},\sf{m})$-bundles as complex analytic families of
Riemann surfaces of genus $\sf{g}$ equipped with $\sf{m}$ distinguished points, for short, as complex analytic (or holomorphic) families of Riemann surfaces of type $({\sf g,m})$,
or holomorphic $(\sf{g},\sf{m})$-families.
\index{family of complex manifolds ! differentiable} \index{family of complex manifolds ! holomorphic} \index{Kodaira}

Kodaira (s. \cite{Kodai}) defines a differentiable family of compact complex manifolds as
a smooth fiber bundle with compact fibers and the following additional structure. The fibers
are equipped with complex structures that depend smoothly on the parameter and induce on
each fiber the smooth fiber structure. The formal definition of a differentiable family of
Riemann surfaces of type $(\sf{g},\sf{m})$ is the following.

\begin{defn}\label{defEl4}
A differentiable
family of Riemann surfaces $\mathcal{M}_t,\, t \in \mathcal{B}$, of type
$(\sf{g},\sf{m})$ is a tuple $(\mathcal{M},\, \mathcal{P},\,\mathbold{E},\, \mathcal{B})$, where $\mathcal{M}$ and $\mathcal{B}$
are oriented $C^{\infty}$ manifolds, and $\mathcal{P}$ is a smooth proper orientation preserving submersion with the
following property. For each $t \in \mathcal{B}$ the fiber $\mathcal{P}^{-1}(t)$, denoted by
$\mathcal{M}_t$, is equipped with the structure of a compact Riemann surface of genus $\sf{g}$,
such that the complex structure
induces the differentiable structure defined on $\mathcal{P}^{-1}(t)$ as a submanifold of
$\mathcal{M}$. $\mathbold{E}\subset \mathcal{M}$ is a smooth submanifold of $\mathcal{M}$ that intersects each fiber $\mathcal{M}_t$ along a set $E_t$ of $\sf{m}$ distinguished points. Moreover, the complex structures of the $\mathcal{M}_t$ depend smoothly on the parameter
$t$, more precisely, there is an open locally finite cover $U_j,\, j=1,2,\ldots$, of
$\mathcal{M}$  and complex valued $C^{\infty}$ functions $(z_1^j,\ldots, z_n^j)$ on each
$U_j$ (with $n$ the complex dimension of $\mathcal{M}_t$) such that for each $j$ and $t$
these functions define holomorphic coordinates on $U_j \cap \mathcal{M}_t$.
 \end{defn}
We will also call such families smooth families of Riemann surfaces of type $(\sf{g},\sf{m})$, or smooth $(\sf{g},\sf{m})$-families.
\index{$\mathcal{M}$} \index{$\mathcal{B}$}

Notice that each smooth special $(0,n+1)$-bundle over a finite oriented smooth surface carries automatically the structure of a smooth family of Riemann surfaces of type $(0,n+1)$.

Two smooth families of Riemann surfaces of type $(\sf{g},\sf{m})$ will be called isomorphic (as families of Riemann surfaces) if they are isomorphic as smooth bundles (smoothly isomorphic, for short) and there is a bundle isomorphism that is holomorphic on each fiber.

We will show in this section that each smooth elliptic fiber bundle with a section is smoothly isomorphic (equivalently, isotopic) to a bundle that carries the structure of a differentiable
family of Riemann surfaces of type $(1,1)$, also called a differentiable family of
closed Riemann surfaces of genus $1$ with a smooth section.

Recall that each compact Riemann surface of genus $1$ is conformally equivalent to the quotient
$\mathbb{C} \diagup \Lambda$ for a lattice $\Lambda=a\mathbb{Z}+b\mathbb{Z}$ where $a$ and
$b$ are two complex numbers that are independent over the real numbers. A torus given in the
form  $\mathbb{C} \diagup \Lambda$ will be called here a canonical torus.
The torus $\mathbb{C} \diagup (\mathbb{Z}+i \mathbb{Z})$ will be called standard.

A family of lattices $\Omega \ni x \to \Lambda(x)$ on a smooth manifold $\Omega$ is called
smooth if for each $x_0 \in X$
there is a neighbourhood $U(x_0) \subset \Omega$ of $x_0$ such that the lattices
$\Lambda (x)$ can be written as
\begin{equation}\label{eqEl10a}
\Lambda(x) = a(x){\mathbb Z}  + b(x){\mathbb Z}\,.
\end{equation}
for smooth functions $a$ and $b$ on $U(x_0)$. If $\Omega$ is a complex manifold,
the family is called holomorphic if it can be locally represented by
\eqref{eqEl10a} with holomorphic functions $a$ and $b$.

If on $U(x_0)$ we have $\widetilde\Lambda (x) = \Lambda (x),\,$ with
$\widetilde\Lambda (x) = \widetilde a(x){\mathbb Z} + \widetilde
b(x){\mathbb Z}$ for other smooth functions $\widetilde a$ and
$\widetilde b$, then there is a matrix $A= \begin{pmatrix} \alpha & \beta \\
\gamma & \delta \end{pmatrix} \in {\rm SL}_2 ({\mathbb Z})\,$ not depending on $x$,
such that
\begin{equation}\label{eqEl2a}
\widetilde a(x) = \alpha \, a(x) \,+\gamma \,b(x),\;\;
\widetilde b(x) = \beta \,a(x) \,+ \delta\, b(x)\,.
\end{equation}
Indeed, $\Lambda(x)$ and $\widetilde{\Lambda}(x)$ are equal lattices
and $(a(x),b(x))$ and $(\widetilde a(x), \widetilde b(x))$ are pairs
of generators of the lattice. Put  $B(x)= \begin{pmatrix} {\rm{Re}}\,a(x) & {\rm{Re}}\,b(x) \\
{\rm{Im}}\,a(x) & {\rm{Im}}b(x) \end{pmatrix}$ and
$\widetilde{ B}(x)= \begin{pmatrix} {\rm{Re}}\,\widetilde a(x) & {\rm{Re}}\,\widetilde b(x) \\
{\rm{Im}}\,\widetilde a(x) & {\rm{Im}} \, \widetilde b(x)
\end{pmatrix}\,.$
Then the real linear self-map of $\mathbb{C}$, defined by $B(x)$ (by
$\widetilde {B}(x)$, respectively) takes $1$ to $a(x)$ and $i$ to $b(x)$ ($1$ to $\widetilde{a}(x)$ and $i$ to $\widetilde{b}(x)$, respectively). Hence, both maps take
the standard lattice $\mathbb{Z} + i \mathbb{Z}$
onto $\Lambda (x) = \widetilde \Lambda (x)$. Then $ B^{-1}(x)\circ\widetilde B (x)$ maps the standard lattice onto itself, hence, it is
in ${\rm SL}_2 ({\mathbb Z})$. Since $B(x)$ and $\widetilde B (x)$
depend continuously on $x$, the matrices $ B ^{-1}(x) \circ \widetilde B(x) $ depend
continuously on $x$ and, hence, are equal to a matrix $A= \begin{pmatrix} \alpha &
\beta \\
\gamma & \delta \end{pmatrix} \in {\rm SL}_2 ({\mathbb Z})\,$ that does not depend on $x$.
Equation \eqref{eqEl2a} is satisfied for the entries of this matrix.

Two smooth families of lattices $\Lambda_0(x)$ and $\Lambda_1(x)$ depending on a parameter $x$ in a
smooth manifold $\Omega$ are called isotopic, if for an interval $I\supset [0,1]$ there is a smooth family of lattices
${\Lambda}(x,t)$,  $(x,t)\in \Omega \times I$, such that
${\Lambda}(x,0)=\Lambda_0(x)$ and  ${\Lambda}(x,1)=\Lambda_1(x)$ for all $x\in \Omega$.

Let $X$ be a finite open Riemann surface and $\Lambda(x),\, x\in X,$ a smooth family of lattices on $X$. Take an
open subset of $U$ of $X$ on which there are smooth function
$a(x)$ and $b(x)$ such that
$\Lambda(x)=a(x)\mathbb{Z}+b(x)\mathbb{Z},\, x\in U$. Then the family $\Lambda= \Lambda(x),\, x \in U,$ defines a free and properly discontinuous group action (see e.g. \cite{Lehn}, III, 3K)
\begin{equation}\label{eqEl11a}
U\times \mathbb{C} \ni (x,\zeta) \to \big(x, \zeta+ a(x)n+ b(x) m\big),\,n,m\in \mathbb{Z}\,,
\end{equation}
on $U\times \mathbb{C}$. The quotient of  $U\times \mathbb{C}$ by this action depends only on $\Lambda$, not on the choice of the generators $a(x)$ and $b(x)$ of the lattice.
Denote by $(U\times \mathbb{C})\diagup \Lambda$ the quotient of $U\times
\mathbb{C}$ by this action \eqref{eqEl11a}. Let
\begin{equation}\label{eqEl11b}
\mathcal{P}_{U,\Lambda}:(U\times \mathbb{C})\diagup \Lambda \to U
\end{equation}
be the mapping whose value at the equivalence class of $(x,\zeta)$ equals  $x$.
Consider in each fiber $\mathcal{P}_{\Lambda}^{-1}(x)$ the distinguished point $\;s_{\Lambda,x}=(x,0)\diagup \Lambda(x)\;$. The mapping $\;x\to s_{\Lambda,x},\, x \in U,\; $ is smooth, hence defines a smooth section. Put $\;\;\;\,\mathbold{s}_{\Lambda,U}\stackrel{def}=\cup_{x\in U}\{s_{\Lambda,x}\}$.
Then
the tuple $\mathfrak{F}_{\Lambda,U}\stackrel{def}= \big((U\times \mathbb{C}) \diagup \Lambda,\mathcal{P}_{U,\Lambda}, \,\mathbold{s}_{\Lambda,U}\,,U\big)$
defines a smooth $(1,1)$-bundle over $U$. Moreover, $\mathfrak{F}_{\Lambda,U}$
is equipped with the structure of a differentiable family of Riemann surfaces of
type $(1,1)$.
If the family $\Lambda$ of lattices on $X$ is holomorphic, then
$\mathfrak{F}_{\Lambda,U}$ is
a holomorphic $(1,1)$-bundle, equivalently, a
holomorphic family of Riemann surfaces of type $(1,1)$. The bundle $\mathfrak{F}_{\Lambda,U}$ depends only on $\Lambda\mid U$, not on the choice of the generators $a(x)$ and $b(x)$ of the lattice $\Lambda(x)$, $x\in U$.

Take an
open cover of $X$ by sets $U_j$ on which there are smooth functions
$a_j(x)$ and $b_j(x)$ such that
$\Lambda(x)=\Lambda_j(x)=a_j(x)\mathbb{Z}+b_j(x)\mathbb{Z},\, x\in U_j$.
Consider the set $\Big\{\big(x,\zeta\big)\diagup \Lambda(x),\, x\in X,\, \zeta \in \mathbb{C}\Big\}$.
The family $(U_j\times \mathbb{C}) \diagup \Lambda_j$ provides a system of smooth  coodinates on this set that are holomorphic for wach fixed value of $x$.
We denote the set  $\big\{(x,\zeta)\diagup \Lambda(x),\, x\in X,\, \zeta \in \mathbb{C}\big\}$ with the thus obtained structure by
$\mathcal{X}_{\Lambda}=
(X\times \mathbb{C})\diagup \Lambda$.
Define a mapping $\mathcal{P}_{\Lambda}$ by the equalities $\mathcal{P}_{\Lambda}=\mathcal{P}_{U_j,\Lambda}$ on $U_j$, and define $\mathbold{s}_{\Lambda}$ by $\mathbold{s}_{\Lambda}\cap  (\mathcal{P}_{\Lambda}^{-1}(U_j))=\mathbold{s}_{\Lambda,U_j}$ for each $j$.
We obtain a smooth  family of Riemann surfaces of type $(1,1)$ on $X$, which we denote by
$\mathfrak{F}_{\Lambda,X}=(\mathcal{X}_{\Lambda},\mathcal{P}_{\Lambda}, \,\mathbold{s}_{\Lambda}\,,X)$.

\index{$\mathfrak{F}_{\Lambda,X}=(\mathcal{X}_{\Lambda},\mathcal{P}_{\Lambda}, \,\mathbold{s}_{\Lambda}\,,X)$}

If the family $\Lambda$ of lattices on $X$ is holomorphic, then
$\mathfrak{F}_{\Lambda,X}$ is
a holomorphic $(1,1)$-bundle, equivalently, a
holomorphic family of Riemann surfaces of type $(1,1)$.

The following lemma holds.
\begin{lemm}\label{lemEl1}
Each smooth $(1,1)$-bundle $\mathfrak{F}$ over a smooth finite open oriented surface $X$ is
smoothly isomorphic (equivalently, isotopic) to a bundle of the form
\begin{equation}\label{eqEl10b}
\mathfrak{F}_{\Lambda,X}=(\mathcal{X}_{\Lambda},\mathcal{P}_{\Lambda}, \,\mathbold{s}_{\Lambda}\,,X)\,.
\end{equation}
\end{lemm}
Recall that the bundle \eqref{eqEl10b} carries the structure of a smooth family of Riemann surfaces of type $(1,1)$.

\medskip

\noindent {\bf Proof.} We need to find a smooth bundle of the form \eqref{eqEl10b} whose monodromy homomorphism is conjugate to that of the original bundle $\mathfrak{F}$. This can be done as in Section \ref{sec:9.2}. We represent $X$ as the union of an open disc $D$
containing the base point of $X$, and a collection of $\ell$  attached
bands $V_j$ that are relatively closed in $X$ and correspond to the generators of the fundamental group of $X$. Let as in Section \ref{sec:9.2} the set $\tilde{U}$ be a simply connected domain on the universal covering of $X$ such that each point of $D$ is covered $\ell+1$ times and each other point of $X$ is covered once. Label the preimages of $D$ under the projection $\tilde{U}\to X$ by $\tilde{D}_j,\; j=0,1,\ldots \ell,$ so that the preimage $\tilde{V}_j$ of the band $V_j,\,j=1,\ldots,\ell,$ is attached to $\tilde{D}_0$ and $\tilde{D}_j$. Denote by $\mathfrak{m}_j$ the monodromy of $\mathfrak{F}$ along the generator $e_j$ of the fundamental group of $X$ that corresponds to the $j$-th band $V_j$.

Map
the fiber of $\mathfrak{F}$ over the base point diffeomorphically onto the standard torus
$\mathbb{C}\diagup(\mathbb{Z}+i \mathbb{Z})$ so that the distinguished point is mapped to $0\diagup (\mathbb{Z}+i \mathbb{Z})$. By an isomorphism between mapping class groups
we identify each monodromy mapping class $\mathfrak{m}_j, \, j=1,\ldots,\ell,$ of $\mathfrak{F}$ with a mapping class on the standard torus
$\mathbb{C}\diagup(\mathbb{Z}+i \mathbb{Z})$ with distinguished point $0\diagup (\mathbb{Z}+i \mathbb{Z})$. We denote the new mapping class by the same letter $\mathfrak{m}_j$.
Represent the new class $\mathfrak{m}_j$
by a mapping
$ \varphi_j^{-1}:\mathbb{C}\diagup(\mathbb{Z}+i \mathbb{Z})\toitself$,                    such
that $\varphi_j$ lifts  to  a real linear self-mapping $\widetilde{\varphi}_j$ of $\mathbb{C}$
that maps the lattice $\mathbb{Z}+i \mathbb{Z}$ onto itself. In other words,
$\widetilde{\varphi}_j$ corresponds to a $2 \times 2$ matrix $A_j$ with
integer entries and determinant $1$, $\widetilde {\varphi}_j (x + iy) = A_j
\begin{pmatrix} x \\ y \end{pmatrix}$, $A_j \in {\rm SL}_2 ({\mathbb
Z})$.

Consider on each set $\widetilde{\Omega}_j\stackrel{def}=\tilde{D}_0\cup \tilde{V}_j\cup \tilde{D}_j$ (which is a simply connected domain)
a smooth family of real linear self-maps $(\widetilde\varphi_j)_z$, $z \in \widetilde{\Omega}_j$, of the
complex plane ${\mathbb C}$ such that
\begin{align}\label{eqEl10a''}
(\widetilde\varphi_j)_z=\mbox{Id} \;\mbox{for $z\in \tilde{D}_0$ and}\,
(\widetilde\varphi_j)_z=\widetilde \varphi_j\; \mbox{for $z\in \tilde{D}_j$}\,,
\end{align}
and put
\begin{equation}\label{eqEl10a'}
\Lambda_j(z)=(\widetilde\varphi_j)_z(\mathbb{Z}+i\mathbb{Z}),\;z \in \widetilde{\Omega}_j\,.
\end{equation}
For points $\tilde{z}_j\in \tilde{D}_j,\, j=0,\ldots \ell,$ that project to the same point in $D$,
the lattices $\Lambda_j(\tilde{z}_j)$ coincide. Hence, there exists a well defined smooth family of lattices $\Lambda(x),\, x\in X,$ that lifts to the family of lattices  \eqref{eqEl10a'} on  $\widetilde{\Omega}_j$.
We obtain a bundle
$\mathfrak{F}_{\Lambda,X}=\big(\,(X\times \mathbb{C})\diagup \Lambda\,,\, \mathcal{P}_{\Lambda}\,,\, \mathbold{s}_{\Lambda}\,,\, X\,\big)$ over $X$, which is isomorphic (as a smooth bundle) to
the original bundle over $X$, since the monodromy homomorphisms of the two bundles are conjugate to each other.
The lemma is proved. \hfill $\Box$
\medskip

By Lemma  \ref{lemEl1} the Problem \ref{problEl.1a}  can be reformulated as follows.

\smallskip

\noindent {\bf Problem $9.1'$} {\it Let $X$ be a finite open Riemann surface. Is a given
smooth family of Riemann surfaces of type $(1,1)$ on $X$
isotopic to a
complex analytic family of Riemann surfaces of type $(1,1)$?}

\section{Complex analytic families of canonical tori.      }
\label{sec:9.5}
Let now ${\mathfrak F} = ({\mathcal X} , {\mathcal P} , X)$ be a
{\it holomorphic} elliptic fiber bundle over a finite open Riemann
surface $X$.
The fiber bundle is, in particular, a smooth elliptic fiber bundle. For each disc $\Delta
\subset X$ there is a smooth family of diffeomorphisms $\varphi_t:S \to
\mathcal{P}^{-1}(t),\; t \in \Delta,$ from the reference Riemann surface $S=\mathbb{C}\diagup (\mathbb{Z}+i \mathbb{Z})$ of genus one onto the fiber over $t$.
Consider the Teichm\"uller class $[\varphi_t],\; t \in \Delta$.
The following Lemma \ref{lemEl3} states that these Teichm\"uller classes depend
holomorphically on the parameter.

\begin{lemm}\label{lemEl3}
Let $\mathfrak{F}$ be a holomorphic elliptic fiber bundle over a
Riemann surface $X$. For each small enough disc $\Delta \subset X$
there is a holomorphic map $z \to \tau (z)$, $z \in \Delta$, into the Teichm\"uller space
$\mathcal{T}(1,0)$ of the standard torus, such that each fiber ${\mathcal P}^{-1}(z) $ is conformally equivalent to $\mathbb{C}\diagup
(\mathbb{Z}+\tau(z)\mathbb{Z})$.
\end{lemm}

For convenience of the reader we will provide a proof. The key ingredient is a lemma of
Kodaira which we formulate now.

Let ${\mathcal X}$ and $X$ be complex manifolds and let ${\mathcal
P} : {\mathcal X} \to X$ be a proper holomorphic submersion such
that the fibers ${\mathcal X}_z = {\mathcal P}^{-1} (z)$, $z\in X$,
are compact complex manifolds of complex dimension $n$. For each $z
\in X$  we denote by $\Theta_z$ the sheaf of germs of holomorphic
tangent vector fields of the complex manifold ${\mathcal X}_z$.
Denote by $H^0 ({\mathcal X}_z , \Theta_z)$ the space of global
sections of the sheaf.

\begin{lemm}\label{lemEl4}
{\rm (Kodaira, \cite{Kodai}, Lemma 4.1, p. 204.)} If the dimension
$d \stackrel{def}= {\rm dim} (H^0 ({\mathcal X}_z , \Theta_z))$ is independent of $z$ then
for any small enough (topological) ball $\Delta$ in $X$ there is for
each $z \in \Delta$ a basis $(v_1 (z) , \ldots , v_d (z))$ of $H^0
({\mathcal X}_z , \Theta_z)$ such that $(v_1 (z) , \ldots , v_d
(z))$  depends holomorphically on $z \in \Delta$.
\end{lemm}
For a proof we refer to \cite{Kodai}.
Let $m$ be the dimension of the complex manifold $X$ in the statement of Kodaira's Lemma. The condition that the $v_j (z)$, $j =
1,\ldots , d$, depend holomorphically on $z$, means the following.
Let $U_{\alpha} \subset {\mathcal X}$ be a small open subset of
${\mathcal X}$ on which there are local holomorphic coordinates
$(\zeta_1^{\alpha} , \ldots , \zeta_n^{\alpha} , z_1 , \ldots ,
z_m)$, where $z = (z_1 , \ldots , z_m) \in \Delta$ and for fixed $z
= (z_1 , \ldots , z_m)$ the $\zeta^{\alpha} = (\zeta_1^{\alpha} ,
\ldots , \zeta_n^{\alpha})$ are local holomorphic coordinates on the
fiber over $z$. In these local coordinates the vector field $v_j$,
$j = 1,\ldots , n$, can be written as
\begin{equation}\label{eqEl17}
v_j (z) = \sum_{k=1}^n v_{j \, k}^{\alpha}
(\zeta^{\alpha} , z) \frac{\partial}{\partial \zeta_k^{\alpha}} \, ,
\end{equation}
where the $v_{j \, k}^{\alpha}$ are holomorphic in $(\zeta^{\alpha}
, z)$. The condition does not depend on the choice of local
coordinates with the described properties.
\index{Kodaira ! Lemma}

We prepare now the proof of Lemma
\ref{lemEl3}.
Kodaira's Lemma applies in the situation of Lemma \ref{lemEl3}. Indeed, in the situation of Lemma \ref{lemEl3}
the base $X$ has complex dimension one and the fibers are compact Riemann surfaces of genus $1$.
The space of holomorphic sections
$H^0 ({\mathcal X}_z , \Theta_z)$ of the sheaf of germs of holomorphic tangent vector
fields of each fiber ${\mathcal X}_z$ has complex dimension one.
Indeed, let $\mathbb{T}$ be a compact Riemann surface of genus $1$. The
complex structure of $\mathbb{T}$ may be given by a system of local holomorphic coordinates
with transition functions $\zeta(z)=z+c$ for complex constants $c$ (so called flat
coordinates). A holomorphic vector field (equivalently, a holomorphic section in the
holomorphic tangent bundle) on $\mathbb{T}$ assigns to each chart  $\mathbb{C} \supset U_j \to V_j \subset \mathbb{T}$ a function $v_j(z_j),\, z_j\in U_j$, such that if $V_j\cap V_k\neq
\emptyset$ then $v_j(z_j(z_k)) \frac{dz_k}{dz_j}=v_k(z_k)$. Since the complex structure on $\mathbb{T}$ is given by a system of flat local
coordinates, the equality $\frac{dz_k}{dz_j}\equiv 1$ holds for each $j,k$ with $V_j\cap V_k\neq \emptyset$.
Hence, the $v_j(z_j)$ define a holomorphic function $v$ on $\mathbb{T}$. Each holomorphic
function on a compact complex manifold is constant. Hence,
the space of holomorphic tangent vector fields to any compact Riemann surface of genus one has complex dimension $1$.
In particular, for each $z \in X$ the space $H^0
({\mathcal X}_z , \Theta_z)$ is generated by a single holomorphic
tangent vector field $v(z)$ on ${\mathcal X}_z$.
By Kodaira's Lemma
$v(z)$ may be chosen to depend holomorphically on $z \in \Delta$ for
each small enough disc $\Delta$ in $X$. We obtain a holomorphic vector field
$v$ on ${\mathcal X}_{\Delta}\stackrel{def}= {\mathcal P}^{-1} (\Delta) $ whose
restriction to each fiber ${\mathcal X}_z$ equals $v(z)$.

Consider the (holomorphic) universal covering space $\widetilde{\mathcal
X}_{\Delta}$ of ${\mathcal X}_{\Delta} $. Denote by $p$ the covering map, $p:\widetilde{\mathcal
X}_{\Delta}\to {\mathcal X}_{\Delta}$.
Consider the triple $(\widetilde{\mathcal
X}_{\Delta}, \mathcal{P}\circ p, \Delta)$.

\begin{lemm}\label{lemEl5}
$(\widetilde{\mathcal X}_{\Delta},\mathcal{P}\circ p, \Delta) $ is holomorphically isomorphic to the trivial
holomorphic fiber
bundle over $\Delta$ with fiber ${\mathbb C}$.
\end{lemm}

\noindent {\bf Proof of Lemma \ref{lemEl5}.} For each $z\in \Delta$ the fiber $\mathcal{P}^{-1}(z)$ of the bundle
$({\mathcal X}_{\Delta},\mathcal{P}, \Delta) $ is
a compact Riemann surface of genus $1$.
Its preimage under $p$ is the set $(\mathcal{P}\circ p)^{-1}(z)$ which is the fiber of the bundle $(\widetilde{\mathcal X}_{\Delta},\mathcal{P}\circ p, \Delta) $ over $z$.
The set $\mathcal{P}^{-1}(z)$
is a complex one-dimensional submanifold of ${\mathcal X}_{\Delta}$ which is the zero set of the holomorphic function $\mathcal{P}-z$.
The set $(\mathcal{P}\circ p)^{-1}(z)$ is a complex one-dimensional submanifold of $\widetilde{\mathcal X}_{\Delta}$ which is the zero set of the holomorphic function
$\mathcal{P}\circ p-z$.
Hence, the restriction $p\mid (\mathcal{P}\circ p)^{-1}(z): (\mathcal{P}\circ p)^{-1}(z)\to \mathcal{P}^{-1}(z)$ defines a holomorphic covering. Indeed, since $p:\widetilde{\mathcal X}_{\Delta}\to {\mathcal X}_{\Delta}$ is a covering, each point $z'$  of $\mathcal{P}^{-1}(z)$ has a neighbourhood $U_{z'}\subset  {\mathcal X}_{\Delta}$ such that $p^{-1}(U_{z'})$ is the disjoint union of open sets $\tilde U _{z'}^k\subset \widetilde{\mathcal{X}}$ and $p$ maps each $\tilde U _{z'}^k$ biholomorphically onto $U_{z'}$. Then $p$  maps each $\tilde U _{z'}^k \cap (\mathcal{P}\circ p)^{-1}(z) $ conformally onto $ U_{z'}\cap \mathcal{P}^{-1}(z)$.

Each fiber $(\mathcal{P}\circ p)^{-1}(z)$ is simply connected, since any loop in the fiber $(\mathcal{P}\circ p)^{-1}(z)$ is contractible in $\widetilde{\mathcal X}_{\Delta}$, and the bundle  $(\widetilde{\mathcal X}_{\Delta},\mathcal{P}\circ p, \Delta) $ is smoothly isomorphic to the trivial bundle $(\Delta\times (\mathcal{P}\circ p)^{-1}(z), {\rm pr}_1, \Delta)$.

Since for each $z\in \Delta$ the covered manifold $\mathcal{P}^{-1}(z) $
is a Riemann surface of genus $1$, the covering manifold $(\mathcal{P}\circ p)^{-1}(z) $ is conformally equivalent to the complex plane $\mathbb{C}$.
Hence, the
triple $(\widetilde{\mathcal X}_{\Delta} , {\mathcal P} \circ p ,
\Delta)$ is a holomorphic fiber bundle with fiber ${\mathbb C}$.

Take a holomorphic section $s(z),\, z\in \Delta,$ of the mapping $\widetilde{\mathcal
X}_{\Delta} \overset{{\mathcal P} \, \circ \,
p}{-\!\!\!-\!\!\!-\!\!\!\longrightarrow} \Delta$. It exists after,
perhaps, shrinking $\Delta$. Let $v$ be the holomorphic
vector field on ${\mathcal X}_{\Delta}$ from Kodaira's Lemma, and
$\widetilde v$ its lift to the universal
covering $\widetilde{\mathcal X}_{\Delta}$. Define a mapping
${\mathcal G} : \Delta \times {\mathbb C} \to \widetilde{\mathcal
X}_{\Delta}$ by
\begin{equation}\label{eqEl18}
{\mathcal G} (z,\zeta) = \gamma_{s(z)} (\zeta) \, ,
\quad (z,\zeta) \in \Delta \times {\mathbb C} \, , 
\end{equation}
where for each $z$ the mapping $\gamma_{s(z)}$ is the solution of
the holomorphic differential equation
\begin{equation}\label{eqEl19}
\gamma'_{s(z)} (\zeta) = \widetilde v (\gamma_{s(z)}
(\zeta)) \, , \quad \gamma_{s(z)} (0) = s(z) \in ({\mathcal P} \circ
p)^{-1} (z) \, , \quad \zeta \in {\mathbb C} \, .
\end{equation}
Since $\widetilde v$ is tangential to the fibers we have the
inclusion
\begin{equation}\label{eqEl20}
\gamma_{s(z)} ({\mathbb C}) \subset ({\mathcal P}
\circ p)^{-1} (z) \cong {\mathbb C} \quad \mbox{for each} \quad z
\in \Delta \, .
\end{equation}
For each $z$ the solution exists for all $\zeta \in {\mathbb C}$
since the restriction of $\widetilde v$ to the fiber over $z$ is the
lift of a vector field on a closed torus. The mapping
\begin{equation}\label{eqEl21}
(z,\zeta) \to \gamma_{s(z)} (\zeta) \, , \quad z \in
\Delta \, , \quad \zeta \in {\mathbb C} \, ,
\end{equation}
is a local holomorphic diffeomorphism.
By the Poincar\'{e}-Bendixson Theorem for each $z$ it maps
${\mathbb C}$ one-to-one onto an open subset of the fiber $({\mathcal P} \circ p)^{-1}
(z) \cong {\mathbb C}$, hence, it maps ${\mathbb C}$ {\it onto}
$({\mathcal P} \circ p) ^{-1} (z)$. Hence, ${\mathcal G}$ defines a
holomorphic isomorphism of the trivial bundle $(\Delta \times\mathbb{C},{\rm pr}_1,
\Delta)$ onto the bundle $(\widetilde{\mathcal
X}_{\Delta},{\mathcal P} \, \circ \,p,\, \Delta)$.
Lemma \ref{lemEl5} is
proved.  \hfill $\Box$

\medskip

\noindent {\bf Proof of Lemma \ref{lemEl3}.}
Consider the covering transformations of the covering
$\widetilde{\mathcal X}_{\Delta} \to {\mathcal X}_{\Delta}$.
In terms of the isomorphic bundle $(\Delta \times \mathbb{C},{\rm pr}_1,
{\mathbb C})$
the restrictions of the covering transformations to each
fiber $\{z\} \times {\mathbb C}$, are translations, hence, the
covering transformations have the form
\begin{equation}\label{eqEl22}
\psi_{n,m} (z,\zeta) =\big(z, \zeta + n \, a(z) + m \,
b(z) \big)\, , \quad (z,\zeta) \in \Delta \times {\mathbb C} \, ,
\end{equation}
for integral numbers $n$ and $m$. Here $a(z)$ and $b(z)$ are complex
numbers which are linearly independent over ${\mathbb R}$ and depend
on $z$. Since the covering transformations $\psi_{1,0}$ and $\psi_{0,1}$ are
holomorphic, the numbers $a(z)$ and $b(z)$ can be taken to depend holomorphically on
$z \in \Delta$. After perhaps interchanging $a(z)$ and $b(z)$ we may assume that
$\mbox{Re}\frac{a(z)}{b(z)}  >0\,.$
Hence, $\tau (z) = \frac{a(z)}{b(z)}$ depends
holomorphically on $z \in \Delta$ and can be interpreted as a holomorphic mapping from
$\Delta$ into the Teichm\"uller space (which is identified with the upper half-plane
$\mathbb{C}_+$). Lemma \ref{lemEl3} is proved.  \hfill $\Box$


\begin{cor}\label{cor1}
Let $\mathfrak{F}$ be a holomorphic elliptic fiber bundle over a connected
simply connected Riemann surface $X$, or let $\mathfrak{F}$ be a holomorphic elliptic fiber bundle over a connected finite open Riemann surface $X$, such that all monodromies are equal to the identity. Then
there is a holomorphic map $z \to \tau (z)$, $z \in X$, into the Teichm\"uller space $\mathcal{T}(1,0)$, such
that each fiber ${\mathcal P}^{-1}(z) $ is conformally equivalent to $\mathbb{C}\diagup
(\mathbb{Z}+\tau(z)\mathbb{Z})$.
\end{cor}

\noindent {\bf Proof.} Let first $X$ be simply connected. Consider a smooth family of diffeomorphisms $\varphi_z: S \to
\mathcal{P}^{-1}(z),\ z \in X$, where $S=\mathbb{C}\diagup (\mathbb{Z}+ i \mathbb{Z})$ is the standard torus. Then the Teichm\"uller
classes $\tau(z) \stackrel{def}{=}[\varphi_z], z \in X,$ depend smoothly on $z$ and
$\mathcal{P}^{-1}(z)$ is conformally equivalent to $\mathbb{C}\diagup
(\mathbb{Z}+[\varphi_z]\mathbb{Z})$. By Lemma \ref{lemEl3}
$[\varphi_z]$ depends holomorphically on $z$. Indeed, on any small enough disc $\Delta
\subset X$ there is a holomorphic map $\tau_{\Delta}$ to the Teichm\"uller space such that
$\mathcal{P}^{-1}(z)$ is conformally equivalent to $\mathbb{C}\diagup (\mathbb{Z}+ \tau_{\Delta}(z)
\mathbb{Z})$. Hence, $[\varphi_z]=\varphi_{\Delta}^*(\tau_{\Delta}(z)),\; z \in \Delta,\;$ for a modular
transformation $\varphi^*_{\Delta}$ of the Teichm\"uller space. The corollary follows from the fact that modular
transformations are biholomorphic selfmaps of the Teichm\"uller space.

Let now $X$ be a finite open Riemann surface, and let all monodromies of the bundle $\mathfrak{F}$ be equal to the identity. Consider a lift $\tilde{\mathfrak{F}}$ of the bundle to the universal covering $\tilde X \overset{\sf P}{-\!\!\!\longrightarrow}    X$. There exists a holomorphic mapping $\tilde{\tau}:\tilde{X}\to \mathcal{T}(1,0)$ such that for each $\tilde{x}\in\tilde{X}$
the value $\tilde{\tau}(\tilde{x})$ represents the conformal class of the fiber of $\tilde{\mathfrak{F}}$ over $\tilde{x}$ (which is equal to the conformal class of the fiber of $\mathfrak{F}$ over ${\sf P}(\tilde{x})$). Let $e$ be an element of the fundamental group of $X$ with base point $x_0$.
We let $\varphi_e$ be a self-homeomorphism of the fiber of $\mathfrak{F}$ over $x_0$, that represents the monodromy mapping class of the bundle along $e$. Denote by $\varphi_e^*$
the modular transformation on $\mathcal{T}(1,0)$ corresponding to $\varphi_e$. Then
$\tilde{\tau}(e(\tilde{x}))=\varphi_e^*(\tilde{\tau}(\tilde{x}))$ for each $e$.
Since all monodromies of the bundle $\mathfrak{F}$ are trivial, each $\varphi_e^*$ is equal to the identity, and therefore $\tilde{\tau}$ descends to a well-defined mapping $\tau$ on $X$ with the required property.
\hfill $\Box$

\medskip

The following proposition holds.

\begin{prop}\label{propEl.2a} Each holomorphic $(1,1)$-bundle
$(\mathcal{X},\mathcal{P},\mathbold{s}, X)$ over a finite open Riemann surface $X$ is holomorphically isomorphic to a
holomorphic bundle of the form
\begin{equation}\label{eqEl11e}
(\mathcal{X}_{\Lambda},\mathcal{P}_{\Lambda}, \,\mathbold{s}_{\Lambda}\,,X)\,.
\end{equation}
for a holomorphic family of lattices $\Lambda=\{\Lambda(x)\}_{x\in X}$.
\end{prop}
Here as before, $\mathcal{X}_{\Lambda}=(X\times \mathbb{C})\diagup \Lambda$, and
$\mathcal{P}_{\Lambda}$ assigns to each class $(x,\zeta)\diagup \Lambda$ in the quotient
 the point $x\in X$.
The value at the point $x\in X$ of the holomorphic section $\mathbold{s}_{\Lambda}$ of the bundle \eqref{eqEl11e}
is the class in the quotient that contains $(x,0)$.
\index{$\mathcal{X}_{\Lambda}$} \index{$\mathcal{P}_{\Lambda}$} \index{$\mathbold{s}_{\Lambda}$}
\medskip

\noindent {\bf Proof.} For each $x\in X$ we let $\Delta_x$ be a topological disc, $x \in\Delta_x \subset X$, for which Kodaira's Lemma holds. As in Lemma \ref{lemEl5} we let
$p_{\Delta_x}:\widetilde{\mathcal{X}}_{\Delta_x} \to {\mathcal{X}}_{\Delta_x} $ be the universal covering of
$\mathcal{X}_{\Delta_x}=\mathcal{P}^{-1}(\Delta_x)$. Let $\tilde{\mathbold{s}}_{\Delta_x}  $ be a lift of ${\mathbold{s}}_{\Delta_x} \stackrel{def}=  \mathbold{s}_{\Lambda}
\cap \mathcal{X}_{\Delta_x}$ to $\widetilde{\mathcal{X}}_{\Delta_x}$.
By Lemma \ref{lemEl5} there exists (after perhaps shrinking $\Delta_x$)
a holomorphic bundle isomorphism
\begin{equation}\label{eqEl06}
\Big(\,\widetilde{\mathcal{X}}_{\Delta_x},\, (\mathcal{P}\circ p_{\Delta_x})| \widetilde{\mathcal{X}}_{\Delta_x}\,, \tilde{\mathbold{s}}_{\Delta_x},\,\Delta_x\,\Big) \to \Big(\,\Delta_x\times \mathbb{C},\, {\rm pr}_1,\,\Delta_x\times \{0\},\, \Delta_x\,\Big)\,.
\end{equation}
The bundle isomorphism is given by the identity mapping on $\Delta_x$ and a holomorphic mapping $\tilde{\Phi}_{\Delta_x}$ from $\widetilde{\mathcal{X}}_{\Delta_x}$ to $\Delta_x\times \mathbb{C}$, which maps the fiber of the first bundle over each point $x'\in \Delta_x$ to the fiber of the second bundle over the same point $x'$ and maps the section of the first bundle to the section of the second bundle.

Cover $X$ by a locally finite set of such discs $\Delta_j$.
For each $\Delta_j$ we consider the covering transformations of the projection $p_{\Delta_j}:\widetilde{\mathcal{X}}_{\Delta_j}\to \mathcal{X}_{\Delta_j}$. The bundle isomorphism
conjugates the group of these covering transformations to a group of fiber preserving transformations of $\Delta_j\times \mathbb{C}$ with free and properly discontinuous action.
The conjugated group acts on each fiber $\{x\}\times \mathbb{C}$ as a lattice
$\Lambda_j(x)$, and the lattices depend holomorphically on $x\in \Delta_j$.
Taking the quotient, we obtain for each $j$ a holomorphic bundle isomorphism
\begin{equation}\label{eqEl07}
\big(\mathcal{X}_{\Delta_j},\, \mathcal{P}_{\Delta_j},\; \mathbold{s}_{\Delta_j},\, \Delta_j\,\big)\to \big((\Delta_j\times \mathbb{C})\diagup {\Lambda}_{j},\; {\rm pr}_1,\, (\Delta_j\times\{0\})\diagup {\Lambda}_{j},\; \Delta_j\,\big).
\end{equation}
The bundle isomorphism  \eqref{eqEl07} is given by the identity mapping from $\Delta_j$ to itself and by a biholomorphic mapping
\begin{align}\label{eqEl07'}
\Phi_j:\mathcal{X}_{\Delta_j}\to
(\Delta_j\times \mathbb{C})\diagup {\Lambda}_{j}
\end{align}
that maps the fiber of the first bundle over any point $x\in \Delta_j$ to the fiber of the second bundle over the same point $x$ and maps the section of the first bundle to the section of the second bundle.

Take $j$ and $k$ so that $\Delta_j\cap \Delta_k\neq \emptyset$.
Then $\Phi_j\circ (\Phi_k)^{-1}$ is a fiber preserving biholomorphic mapping from  $((\Delta_j\cap \Delta_k)\times \mathbb{C})\diagup {\Lambda}_{k}$ onto
$((\Delta_j\cap \Delta_k)\times \mathbb{C})\diagup {\Lambda}_{j}$.
Moreover, $\Phi_j\circ (\Phi_k)^{-1}$ maps $((\Delta_j\cap \Delta_k)\times\{0\})\diagup {\Lambda}_{k}$ onto $((\Delta_j\cap \Delta_k)\times\{0\})\diagup {\Lambda}_{j}$.
The mapping $\Phi_j\circ (\Phi_k)^{-1}$ lifts to a fiber preserving biholomorphic mapping  $A_{j,k}$
of $ (\Delta_j\cap \Delta_k)\times \mathbb{C}$ onto itself, that takes $(\Delta_j\cap \Delta_k)\times\{0\}$ onto itself. Then
\begin{equation}\label{eqEl11j}
A_{j,k}(x,\zeta)= (x,\alpha_{j,k}(x)\cdot \zeta),\; (x,\zeta)\in (\Delta_j\cap \Delta_k)\times \mathbb{C}\,
\end{equation}
for a nowhere vanishing holomorphic function $\alpha_{j,k}$ on $(\Delta_j\cap \Delta_k)$.
Moreover, the mapping $A_{j,k}$
takes each $\{x\}\times {\Lambda}_k(x)\subset \{x\}\times \mathbb{C}$ onto $\{x\}\times {\Lambda}_j(x)\subset \{x\}\times \mathbb{C}$.

The $\alpha_{j,k}$ form a Cousin II cocycle on $X$ (see \cite{H1} and Appendix \ref{ChapterA}). Since $X$ is a finite open Riemann surface, the Cousin II problem is solvable (see \cite{Fo}), i.e. there exist holomorphic nowhere vanishing functions $\alpha_j$ on $\Delta_j$, such that on $(\Delta_j\cap \Delta_k)$ the equality
$$
\alpha_{j,k}=\alpha_j\cdot \alpha_k^{-1}
$$
holds.
Then for each $j,k$ with $\Delta_j\cap \Delta_k\neq \emptyset$ the equality
$(\alpha_j(x))^{-1}\Lambda_j(x)=(\alpha_k(x))^{-1}\Lambda_k(x),\; x\in \Delta_j\cap \Delta_k$ holds.
Hence, we obtain a well-defined holomorphic family of lattices $\Lambda$ on $X$, $\Lambda(x)=(\alpha_j(x))^{-1}\Lambda_j(x),\, x \in \Delta_j$.
For this family of lattices we consider the bundle
$\big((X\times\mathbb{C})\diagup \Lambda, \mathcal{P}_{\Lambda}, \mathbold{s}_{\Lambda},X\big)$
of the form \eqref{eqEl11e}.

Let $c_j^{-1}:(\Delta_j\times \mathbb{C})\diagup {\Lambda}_j\to (\Delta_j\times \mathbb{C})\diagup {\Lambda}$ be the biholomorphic mapping, that lifts to the biholomorphic self-mapping $(x,\zeta)\to(x,( \alpha_j(x))^{-1}\,\zeta)$ of $\Delta_j\times \mathbb{C}$, which maps each $\{x\}\times\Lambda_j(x)$ onto $\{x\}\times
(\alpha_j(x))^{-1}\Lambda_j(x)=\{x\}\times\Lambda(x)$.

Put $\Phi'_j= c_j^{-1} \Phi_j$ on $\mathcal{X}_{\Delta_j}$. Then on each non-empty intersection $(\Delta_j\cap \Delta_k)$ the equality $\Phi'_j (\Phi'_k)^{-1}=c_j^{-1}\Phi_j (c_k^{-1}\Phi_k)^{-1}= c_j^{-1} (\Phi_j (\Phi_k)^{-1})c_k$
holds. The latter mapping lifts to
\begin{align*}
(x,\zeta)\to (x, \alpha_j(x)^{-1}\alpha_{j,k}(x)\alpha_k(x)\zeta)=(x,\zeta),\;\;\; (x,\zeta) \in (\Delta_j\cap \Delta_k)\times \mathbb{C}\,.
\end{align*}
Hence, we obtain
$\Phi'_j=\Phi'_k$ on $\mathcal{X}_{\Delta_j}\cap \mathcal{X}_{\Delta_k}$. The mapping $\Phi':\mathcal{X}\to (X\times \mathbb{C})\diagup \Lambda$, for which $\Phi'(x)=\Phi'_j(x)$ for $x\in \mathcal{X}_{\Delta_j}$, is well defined and determines a holomorphic isomorphism from the original bundle to a bundle of the form \eqref{eqEl11e}. Proposition \ref{propEl.2a} is proved. \hfill $\Box$

\smallskip

Similar arguments as used in the proof of Proposition \ref{propEl.2a} give the following statement.\\
\smallskip

\noindent {\bf Proposition $9.1'$.} {\it Each holomorphic elliptic fiber bundle over a finite open Riemann surface $X$ is holomorphically isomorphic to a holomorphic bundle that admits a holomorphic section.}\\

\smallskip

\noindent {\bf Proof.} Cover $X$ by small discs $\Delta_j$ on which a holomorphic section $\mathbold{s}_{\Delta_j}$ of the  original bundle can be chosen.
For each $j$ we consider the holomorphic bundle isomorphism \eqref{eqEl07},
and obtain a holomorphic isomorphism ${\Phi}_j$ of the total space of the bundle on the left onto the total space of the bundle on the right of \eqref{eqEl07}, that maps fibers of the first bundle to fibers of the second one
and the chosen section $\mathbold{s}_{\Delta_j}$ of the first bundle to the section $(\Delta_j\times\{0\})\diagup {\Lambda}_{j}$ of the second one.

The isomorphism ${\Phi}_j: \mathcal{X}_{\Delta_j}\to(\Delta_j\times \mathbb{C})\diagup \Lambda_j  $ lifts to a holomorphic isomorphism  ${\tilde\Phi}_j: \widetilde{\mathcal{X}}_{\Delta_j}\to\Delta_j\times \mathbb{C} $, that maps a lift $\tilde{\mathbold{s}}_{\Delta_j}$
of $\mathbold{s}_{\Delta_j}$ to ${\Delta_j}\times \{0\}$.
For $\Delta_j\cap \Delta_k\neq \emptyset$ the mapping
${\Phi}_{j,k}={\Phi}_j\circ ({\Phi}_k)^{-1}$ on $\big((\Delta_j\cap \Delta_k)\times \{\mathbb{C}\}\big)\diagup \Lambda_k$
lifts to a fiber preserving biholomorphic mapping $\tilde{\Phi}_{j,k}=\tilde{\Phi}_{j}\circ \tilde{\Phi}_{k}^{-1} $
of $ (\Delta_j\cap \Delta_k)\times \mathbb{C}$ onto itself.
In contrast to the situation in Proposition \ref{propEl.2a} we have no control about the image of  $\big((\Delta_j\cap\Delta_k)\times\{0\}\big)\diagup {\Lambda}_{k}$
under the mapping ${\Phi}_{j,k}={\Phi}_j\circ ({\Phi}_k)^{-1}$. Instead of equation
\eqref{eqEl11j} we obtain
the equation
\begin{equation}\label{eqEl010}
\tilde{\Phi}_{j,k}(x,\zeta)= (x, a_{j,k}(x) + \alpha_{j,k}(x)\cdot \zeta),\; (x,\zeta)\in (\Delta_j\cap \Delta_k)\times \mathbb{C}\,
\end{equation}
for a holomorphic function $a_{j,k}$ and a nowhere vanishing holomorphic function $\alpha_{j,k}$ on $\Delta_j\cap \Delta_k$.
Since $\tilde{\Phi}_{j,k}\,\tilde{\Phi}_{k,i}\, \tilde{\Phi}_{i,j}={\rm Id}$ on $\Delta_j\cap \Delta_k\cap \Delta_i $, we obtain
\begin{align}\label{eqEl011}
\Big(a_{j,k}(x) +\alpha_{j,k}(x) a_{k,i}(x) +\alpha_{j,k}(x)\alpha_{k,i}(x) a_{i,j}(x)\Big)\nonumber\\
+ \Big(\alpha_{j,k}(x)\alpha_{k,i}(x)\alpha_{i,j}(x)\Big)\cdot \zeta \,=\,\zeta,\;\quad
x \in \Delta_j\cap \Delta_k\cap \Delta_i,\; \zeta \in \mathbb{C}\,.
\end{align}
Hence, $\alpha_{j,k}(x)\alpha_{k,i}(x)\alpha_{i,j}(x) \equiv 1$, in other words, the $\alpha_{j,k}$ form a Cousin II cocycle. Since $X$ is an open Riemann surface the Cousin Problem has a solution, i.e. there exist nowhere vanishing holomorphic functions $\alpha_j$ on $\Delta_j$ such that $\alpha_{j,k}=\frac{\alpha_j}{\alpha_k}$ on $(\Delta_j\cap \Delta_k)$.
Then by equation \eqref{eqEl011}
\begin{align*}
\frac{ a_{j,k}}{\alpha_j} + \frac{a_{k,i}}{\alpha_k}  + \frac{a_{i,j}}{\alpha_i} =0\,.
\end{align*}
In other words, the $\frac{ a_{j,k}}{\alpha_j}$ form a Cousin I cocycle. Since $X$ is a Stein manifold the Cousin Problem is solvable, i.e. there exist
holomorphic functions $a_j$ on $\Delta_j$ such that
\begin{align}\label{eqEl012}
\frac{a_{j,k}}{\alpha_j} =\frac{a_j}{\alpha_j}-\frac{a_k}{\alpha_k}
\end{align}
on $(\Delta_j\cap \Delta_k)$.
Consider for all $j$ the mapping $(x,\zeta)\to A_j(x,\zeta)=(x,a_j(x)+\alpha_j(x)\zeta)$. The inverse mapping has the form $A_j^{-1}(x,\zeta)=
\big(x, -\frac{a_j(x)}{\alpha_j(x)}+\frac{1}{\alpha_j(x)}\zeta\big)$. Then by equation \eqref{eqEl012}
we obtain
\begin{align}\label{eqEl012'}
(A_j \, A_k^{-1})(x,\zeta)=  &\Big(x,a_j(x) +\alpha_j(x)(-\frac{a_k(x)}{\alpha_k(x)}+\frac{\zeta}{\alpha_k(x)})\Big)\nonumber\\ =&\Big(x,a_{j,k}(x) +\alpha_{j,k}(x) \zeta\Big)=\tilde{\Phi}_{j,k}(x,\zeta)\,.
\end{align}
Consider the biholomorphic mapping
$\tilde{\Phi}_j^*:\tilde{\mathcal{X}}_{\Delta_j}\to \Delta_j\times \mathbb{C}$ which is the composition $\tilde{\Phi}_j^*=A_j^{-1}\circ\tilde{\Phi}_j$ of the mapping  $A_j^{-1}:\Delta_j\times \mathbb{C}\, \toitself$ with $\tilde{\Phi}_j$.
By equations \eqref{eqEl012'} and \eqref{eqEl010} we obtain
\begin{align*}
\tilde{\Phi}_j^*(\tilde{\Phi}_k^*)^{-1} (x,\zeta) =  A_j^{-1} \tilde{\Phi}_j \tilde{\Phi}_k^{-1} A_k(x,\zeta)= A_j^{-1} \tilde{\Phi}_{j,k}( A_k A_j^{-1}) A_j(x,\zeta)
= (x,\zeta)\,.
\end{align*}
We obtained biholomorphic mappings $\tilde{\Phi}_j^*:\tilde{\mathcal{X}}_{\Delta_j}\to \Delta_j\times \mathbb{C}$, such that on $\Delta_j\cap \Delta_k$ the mappings $\tilde{\Phi}_j^*$ and $\tilde{\Phi}_k^*$ coincide.

Each ${\mathcal{X}}_{\Delta_j}$ is the quotient of $\tilde{\mathcal{X}}_{\Delta_j}$ by the group of covering transformations of the covering $p_{\Delta_j}: \tilde{\mathcal{X}}_{\Delta_j}\to\mathcal{X}_{\Delta_j}   $.
Conjugating for each $j$ the group of covering transformations
by the mapping $\tilde{\Phi}_j^*$, we obtain a group acting on
$\Delta_j\times \mathbb{C}$
which can be identified with a holomorphic family of lattices $\Lambda_j$ on $\Delta_j$.
The mapping $\tilde{\Phi}_j^*$ descends to a biholomorphic mapping $\Phi^*_j:\mathcal{X}_{\Delta_j}\to (\Delta_j\times \mathbb{C})\diagup \Lambda_j$. Since on $\Delta_j\cap \Delta_k$ the mappings $\tilde{\Phi}_j^*$ and $\tilde{\Phi}_k^*$ coincide,
the lattices $\Lambda_j$ and $\Lambda_k$ coincide on  $\Delta_j\cap\Delta_k$, and $\Phi^*_j=\Phi^*_k$ on $\Delta_j\cap\Delta_k$.
For the lattice $\Lambda$ on $X$ that equals $\Lambda_j$ on $\Delta_j$
we obtain a bundle
$((X\times \mathbb{C})\diagup \Lambda,\mathcal{P}_{\Lambda}, \,\mathbold{s}_{\Lambda}\,,X)$
of the form \eqref{eqEl11e}, and the mappings $\Phi_j$ define a holomorphic bundle isomorphism from the original bundle to $((X\times \mathbb{C})\diagup \Lambda,\mathcal{P}_{\Lambda},X)$. The proposition is proved. \hfill $\Box$

\section{Special $(0,4)$-bundles and double branched coverings}
\label{sec:9.6}
In Section \ref{sec:9.8} we will study the Gromov-Oka Principle for elliptic fiber bundles, or, in other words, for families of Riemann surfaces of type $(1,1)$ (see Problem \ref{problEl.1a} and Problem $9.1'$). This will be done by a reduction to the  Gromov-Oka Principle in the case of special $(0,4)$-bundles, which we will prepare now.
We will represent tori with a distinguished point as double branched coverings over $\mathbb{P}^1$ with a set of four branch points. This is done as follows.

Take any element $\mathring{E}=\{z_1,z_2,z_3\}\in C_3(\mathbb{C})\diagup \mathcal{S}_3$, and put $E\stackrel{def}=\mathring{E}\cup \{\infty\}$. The set
\begin{equation}\label{eqEL31a}
\mathring{Y}_E\stackrel{def}= \Big\{(z,w)\in \mathbb{C}^2: w^2 = 4 \, (z - z_1) (z-z_2)
(z-z_3)\Big\}
\end{equation}
is a one-dimensional complex submanifold of $\mathbb{C}^2$.
Each point of $\mathring{Y}_E$ has a neighbourhood in $\mathring{Y}_E$ on which one of the functions, $z$ or $w$, defines local holomorphic coordinates.
The mapping
\begin{equation}\label{eqEl31d}
\mathring{Y}_E\ni (z,w)\to z \in \mathbb{C}
\end{equation}
is a branched holomorphic covering of $\mathbb{C}$ with branch locus $\mathring{E}$.
\index{$\mathring{Y}_E$}

Consider the $1$-point compactification $Y_E$ of $\mathring{Y}_E$. The complex structure on it is obtained as
follows. Let
$\mathring{Y}_E^{\,r}$ be the subset of $\mathring{Y}_E$ where $|z|>r$ for a large positive number $r$. On this set $|\frac{z}{w}|$ is small. Put $(\tilde{z},\tilde{w})=(\frac{1}{z},
\frac{z}{w})$ on a small neighbourhood in $\mathbb{C}^2$ of
$\mathring{Y}_E^{\,r}$. In these coordinates
the equation for $\mathring{Y}_E^{\,r}$ becomes
\begin{equation}\label{eqEl36b}
\tilde{w}^2= \tilde{z} \frac{1}{4 \prod_{j=1}^3(1- z_j\tilde{z})}\,,
\end{equation}
\index{$\mathring{Y}_E$}
and $\mathring{Y}_E^{\,r}$ can be identified with the subset
\begin{align}\label{eqEl36a}
\left\{(\tilde{z},\tilde{w})\in \mathbb{C}^2,\, 0<|\tilde{z}|<\frac{1}{r}:\tilde{w}^2=
\tilde{z} \frac{1}{4 \prod_{j=1}^3(1-z_j\tilde{z})}\right\}
\end{align}
of $\mathbb{C}^2$. Adding the point $(\tilde{z},\tilde{w})=(0,0)\in \mathbb{C}^2$ to
$\mathring{Y}_E^{\,r}$  we obtain a complex manifold ${Y}_E^{\,r}$.
The two manifolds ${Y}_E^{\,r}$  and $\mathring{Y}_E$
form an open cover of the desired compact complex manifold $Y_E$.
Denote the point $(0,0)$ in coordinates $(\tilde{z},\tilde{w})$ on $Y^r_{E}$ by $s^{\infty}$.
Each point of ${Y}_E$ has
a neighbourhood where one of the functions $z$, $w$, $\tilde{z}$, or
$\tilde{w}$ defines local holomorphic coordinates. The holomorphic projection $Y_E\to
\mathbb{P}^1$ is correctly defined by $(z,w)\to z,\, (z,w)\in \mathring{Y}_E,\,$
and $(\tilde{z},\tilde{w})\to  \tilde{z},\,  (\tilde{z},\tilde{w})\in {Y}_E^{\,r}$.
We obtain a double branched covering  $Y_E\to \mathbb{P}^1$
over $\mathbb{P}^1$ with
branch locus equal to $E\stackrel{def}=\mathring{E} \cup \{\infty\}$. The manifold $Y_E$ is a closed Riemann surface of genus $1$. We will consider
it as a closed Riemann surface of genus $1$ with distinguished point being the preimage $s^{\infty}$ of $\infty$ under the branched
covering.
The set $\mathring{E}$ (considered as subset of $\mathbb{C}$) will be called the finite branch locus of
the covering $Y_E\to \mathbb{P}^1$.

Let $Y$ be a closed Riemann surface of genus $1$ with distinguished point $s$ and let $Y_1$ be equal to $\mathbb{P}^1$ with
set of distinguished points $\mathring{E}\cup \{\infty\}$ for $\mathring{E}\subset C_3(\mathbb{C})\diagup
\mathcal{S}_3$. Suppose ${\sf{Pr}}:Y\to Y_1$ is a double branched covering with branch
locus $\mathring{E}\cup \infty$ and ${\sf{Pr}}(s)=\infty$. A mapping class $\mathfrak{m}\in
\mathfrak{M}(Y;s,\emptyset)$ is called a lift of a mapping class $\mathfrak{m}_1\in
\mathfrak{M}(Y_1;\infty,\mathring{E})$ if there are representing homeomorphisms $\varphi \in
\mathfrak{m}$ and $\varphi_1 \in \mathfrak{m}_1$, such that $\varphi$ lifts $\varphi_1$,
i.e. $\varphi_1({\sf{Pr}}(\zeta))={\sf{Pr}}(\varphi(\zeta))$, $\zeta\in Y$.

We define double branched coverings of smooth families of Riemann surfaces with distinguished points as follows.
\begin{defn}\label{defnEl5}
Let $X$ be an oriented smooth surface (a Riemann surface, respectively). Suppose $\mathfrak{F}=\big(\mathcal{X},\mathcal{P}, \mathbold{s},X\big)$ is a smooth
 (complex analytic, respectively) family of Riemann
surfaces of type $(1,1)$ over $X$. Let  $\mathbold{E} \subset X\times \mathbb{P}^1$ be a smooth (complex, respectively) submanifold of $X\times \mathbb{P}^1$, that intersects each fiber
$\{x\}\times\mathbb{P}^1 $ along a set of distinguished points $E_x= \{x\}\times (\mathring{E}_x\cup \{\infty\})$ with  $\mathring{E}_x\subset C_3(\mathbb{C})\diagup \mathcal{S}_3$.

The family $\mathfrak{F}$ is called a double branched covering of the special
 smooth (holomorphic, respectively)
$(0,4)$-bundle $\big(X\times \mathbb{P}^1,{\rm pr}_1,\mathbold{E} ,\,X\big)$
if there exists
a smooth (holomorphic, respectively) mapping ${\sf{Pr}}: \mathcal{X} \to   X\times\mathbb{P}^1$
that maps each fiber $\mathcal{P}^{-1}(x)$ of the $(1,1)$-family $\mathfrak{F}$
onto the fiber
$\{x\}\times \mathbb{P}^1$
of the $(0,4)$-bundle over the same point $x$, such that the
restriction
${\sf{Pr}}: \mathcal{P}^{-1}(x) \to \{x\}\times \mathbb{P}^1$ is a holomorphic double
branched covering with branch locus being the set $\{x\}\times (\mathring{E}_x\cup \{\infty\})$ of distinguished
points in the fiber $\{x\}\times \mathbb{P}^1$, and $\sf{Pr}$
maps the distinguished point $s_x=\mathbold{s}\cap \mathcal{P}^{-1}(x)$  in the fiber $\mathcal{P}^{-1}(x)$ over $x$ to
the point $\{x\}\times \{\infty\}$ in the fiber $\{x\}\times \mathbb{P}^1$ of the special $(0,4)$-bundle.
\end{defn}
We will also write
$(X\times \mathbb{P}^1  , {\rm pr}_1  ,\mathbold{E}, X)={\sf{Pr}}((\mathcal{X},\mathcal{P},\mathbold{s},X))$, and call
the family $(\mathcal{X},\mathcal{P},\mathbold{s}, X)$ a lift of $(X\times\mathbb{P}^1, {\rm pr}_1 ,\mathbold{E} ,X)$.
\index{covering ! double branched covering of fibre bundles}

\begin{lemm}\label{lemEl6}
Each special holomorphic 
$(0,4)$-bundle $\big(X\times \mathbb{P}^1,{\rm pr}_1,\mathbold{E} ,\,X\big)$
over a Riemann surface  $X$ admits a double
branched covering by a complex analytic family of Riemann surfaces of type $(1,1)$.
Each special smooth
$(0,4)$-bundle over an oriented differentiable manifold $X$ has a double branched covering by a
differentiable family of Riemann surfaces of type $(1,1)$.
\end{lemm}
Notice that the statement for the smooth case is true also for families over products $X\times I$ where $X$ is an oriented surface and $I$ is an interval, in other words, it is true for isotopies of families of complex manifolds over Riemann surfaces or orientble smooth surfaces.

\noindent {\bf Proof.} We prove the statement for the holomorphic case. In the smooth case the dependence on the variable $x$ is only smooth, otherwise
the smooth case is treated similarly as the holomorphic case.
Assume the sets $\mathring{ E}_x,\, x\in X,$ (with $\mathbold{E}\cap (\{x\}\times \mathbb{P}^1)=\{x\}\times (\mathring{E}_x \cup \{\infty\})$) are uniformly
bounded in $C_3(\mathbb{C})\diagup\mathcal{S}_3$.
Consider the set
\begin{equation}\label{eqEl37}
\mathring{\mathcal Y}_{\mathbold{E}} = \left\{ (x,z,w) \in X
\times {\mathbb C}^2 : w^2 = 4 \prod_{z_j \in \mathring{E}_x} (z
- z_j) \right\} \,,
\end{equation}
equipped with the structure of an embedded complex hypersurface  in $X\times\mathbb{C}^2$.
In a neighbourhood of a point $(x_0 , z_0 , w_0)$ on
$\mathring{\mathcal Y}_{\mathbold{E}}$ with $z_0 \notin \mathring{E}_{x_0}$ the pair $(x,z)$
defines holomorphic coordinates. If $z_0 \in \mathring{E}_{x_0}$ the pair
$(x,w)$ defines holomorphic coordinates in a neighbourhood of $(x_0
, z_0 , w_0)$ on $\mathring{\mathcal Y}_{\mathbold{E}}$.
The projection ${\mathcal P} :\mathring{\mathcal Y}_{\mathbold{E}} \to X$, ${\mathcal P}
(x,z,w) = x$ is holomorphic.

For some large positive number $r$ we may define the set
\begin{align}\label{eqEl36c}
{\mathcal Y}_{\mathbold{E}}^{\,r}\stackrel{def}=\left\{(x,\tilde{z},\tilde{w})\in X\times
\mathbb{C}^2,\, |\tilde{z}|<\frac{1}{r}:\tilde{w}^2= \tilde{z} \frac{1}{4\prod_{z_j \in \mathring{E}_x}
(1-{z}_j\tilde{z})}\right\}\,.
\end{align}
It is a complex hypersurface of complex dimension two of $X\times \mathbb{C}^2$.
Each of its points has a neighbourhood on which either $(x,\tilde{z})$ or
$(x,\tilde{w})$ defines holomorphic coordinates. The mapping $(x,\tilde{z},\tilde{w})\to x$ is
holomorphic.
The part
$\mathring{\mathcal{Y}}_{\mathbold{E}}^{\,r}\stackrel{def}=\big\{(x,\tilde{z},\tilde{w})
\in{\mathcal{Y}}_{\mathbold{E}}^{\,r}:
\tilde{z}\neq 0\big\}$ of ${\mathcal Y}_{\mathbold{E}}^{\,r}$ can be identified with the subset
$\big\{(x,z,w) \in \mathring{\mathcal{Y}}_{\mathbold{E}}: \, |z|>r\big\}$  of
$\mathring{\mathcal{Y}}_{\mathbold{E}}$ using the transition functions
$(x,z,w)\to (x,\tilde{z} ,\tilde{w})=(x,\frac{1}{z}, \frac{z}{w})$.
The sets $\mathring{\mathcal{Y}}_{\mathbold{E}}$ and ${\mathcal Y}_{\mathbold{E}}^{\,r}$
form an open cover of a complex manifold denoted by ${\mathcal Y}_{\mathbold{E}}$ equipped
with a proper holomorphic submersion $\mathcal{P}:{\mathcal Y}_{\mathbold{E}}\to X$, such that $\mathcal{P}^{-1}(x)$ is a torus for each $x\in X$.
We proved that $({\mathcal Y}_{\mathbold{E}},\mathcal{P},X)$ is a holomorphic elliptic fiber bundle. Let $\mathbold{s}^{\infty}$ be the submanifold of ${\mathcal Y}_{\mathbold{E}}$ that intersects
each fiber $\mathcal{P}^{-1}(x)$ along the
distinguished point $s_x^{\infty}\in \mathcal{Y}^r_{\mathbold{E}}$ that is written in coordinates on $\mathcal{Y}^r_{\mathbold{E}}$ as $(x,0,0)$.

Let ${\sf Pr}: \mathcal{Y}_{\mathbold{E}}\to X\times \mathbb{P}^1$ be the map that assigns to $(x,z,w)$ the point $(x,z)$ (and to  $(x,\tilde{z},\tilde{w})$ the point $(x,\tilde{z})$).
By the construction it is clear that the obtained $(1,1)$-bundle $({\mathcal
Y}_{\mathbold{E}},\mathcal{P},\,\mathbold{s}^{\infty},\,X)$ is a double branched covering of the $(0,4)$-bundle $\big(X\times \mathbb{P}^1,{\rm pr}_1,\mathbold{E} ,\,X\big)$ whose set of finite distinguished points in
the fiber over $x$ is equal to the finite branch locus $\{x\}\times \mathring{E}_x$. The
double branched covering map ${\sf{Pr}}: {\mathcal Y}_{\mathbold{E}}\to X\times \mathbb{P}^1$
maps the distinguished point $s_x^{\infty} \in \mathcal{P}^{-1}(x)$ to the point $\{x\}\times \infty$.
\index{${\sf Pr}$}

In the general case (i.e. without the assumption that the sets $\mathring{E}_x,\,x \in X$, are
uniformly bounded) the statement is proved by considering an exhaustion of $X$ by
relatively compact open sets. \hfill $\Box$

Let $\big(X\times \mathbb{P}_1, {\rm pr}_1, \mathbold{E}_j,X\big)$, $j=0,1,$ be
two special holomorphic (smooth, respectively) $\,(0,4)$-bundles that are isotopic through smooth special $\,(0,4)$-bundles.
Then the
families $({\mathcal Y}_{\mathbold{E}_j},\mathcal{P}_j, \mathbold{s}_j^{\infty},\,X),\, j=0,1,$ of Riemann surfaces of type $(1,1)$ are isotopic.
Indeed, let $I$ be an open interval containing $[0,1]$.
The isotopy of the special $(0,4)$-bundles is given by a bundle
$$
\big((I\times X)\times \mathbb{P}^1, {\rm pr}_1, {\mathbold{E}},I\times X\big)
$$
over $I\times X$.
Here ${{\mathbold{E}}}$
 is a smooth submanifold of $(I\times X)\times \mathbb{P}^1$ such that for each $(t,x)\in I\times X$ the intersection of the set ${\mathbold{E}}$ with the fiber over $(t,x)$ equals $\{(t,x)\}\times {E}(t,x)$ for subsets ${E}(t,x)$ that are the union of
the point $\infty$ with an element of $C_3(\mathbb{C}\diagup \mathcal{S}_3)$. Moreover, ${{\mathbold{E}}}\cap ((\{j\}\times X) \times \mathbb{P}^1)=\mathbold{E}_j,\,j=0,1$.
The smooth special $(0,4)$-bundle $\big((I\times X)\times \mathbb{P}^1, {\rm pr}_1, {\mathbold{E}},I\times X\big)$ has a double branched covering by a smooth family of Riemann surfaces of type $(1,1)$ which defines the required isotopy for the families $({\mathcal Y}_{\mathbold{E}_j},\mathcal{P}_j, \mathbold{s}_j^{\infty},\,X),\, j=0,1,$ of Riemann surfaces of type $(1,1)$.

Notice that for a non-contractible oriented finite open smooth surface $X$ and a
smooth special $(0,4)$-bundle $\big(X\times \mathbb{P}^1, {\rm pr}_1, \mathbold{E},X\big)$ over $X$ the obtained double branched covering
$({\mathcal Y}_{\mathbold{E}},\mathcal{P},\,\mathbold{s}^{\infty},\,X)$ is not the only double branched covering of the $(0,4)$-bundle (see Section \ref{sec:9.7}). We will call the bundle $({\mathcal Y}_{\mathbold{E}},\mathcal{P},\,\mathbold{s}^{\infty},\,X)$ the
canonical double branched covering of the special $(0,4)$-bundle
 $\big(X\times \mathbb{P}^1, {\rm pr}_1, \mathbold{E},X\big)$.

\section{Lattices and double branched coverings}
\label{sec:9.7}
\noindent {\bf The Weierstra\ss\ $\wp$-function.}
Consider the quotient $\mathbb{C}\diagup \Lambda$ of the complex plane by a lattice $\Lambda$. We want to associate to the quotient
a double branched covering over the Riemann sphere with covering space being conformally
equivalent to $\mathbb{C}\diagup \Lambda$. The standard tool for this purpose is
the Weierstra\ss\ $\wp$-function  $\wp_{\Lambda}$.
Put $\Lambda\stackrel{def}=a {\mathbb Z} + b{\mathbb Z}$, where $a$ and $b$ are real linearly independent complex numbers.
The Weierstra\ss\
$\wp$-function,\index{Weierstra\ss \ $\wp$-function}
\begin{equation}\label{eqEl30}
\wp_{\Lambda} (\zeta) = \frac1{\zeta^2} + \sum_{(n,m)
\in {\mathbb Z}^2 \atop (n,m) \ne (0,0)} \left( \frac1{(\zeta -a\, n -b
\,m )^2} - \frac1{(a\,n+b\,m )^2} \right) , \quad \zeta \in
{\mathbb C} \setminus \Lambda\,,
\end{equation}
is meromorphic on ${\mathbb C}$, has poles
of second order at points of $ \Lambda$ and
is holomorphic on ${\mathbb C}\setminus \Lambda$. It has
periods $a$ and $b$ and
principal part $\zeta \to \frac1{\zeta^2}$ at $0$.
The summation is over all non-zero elements of the lattice. Hence, the function depends
only on the lattice, not on the choice of the generators $a$ and $b$ of the lattice.
\index{$\wp_{\Lambda}$}

The Weierstra\ss\ $\wp$-function satisfies the following differential equation
\begin{equation}\label{eqEl30a'}
(\wp'_{\Lambda})^2 (\zeta)=4 (\wp_{\Lambda})^3(\zeta)-g_2(\Lambda) (\wp_{\Lambda})(\zeta)-g_3(\Lambda)
\end{equation}
for complex numbers $g_2(\Lambda)$ and $g_3(\Lambda)$ that depend only on the lattice.
This can be proved using the Laurent series expansion near zero of $\wp_{\Lambda}$ and $
\wp'_{\Lambda}$ and taking a linear combination that has no pole and is therefore constant. For details see \cite{Ap}. The numbers $g_2(\Lambda)$ and $g_3(\Lambda)$ are called the elliptic invariants of $\Lambda$.

Denote the zeros of the equation $\;4t^3-g_2(\Lambda)t-g_3(\Lambda)\,=\,0\;$ by
$e_1(\Lambda), e_2(\Lambda), e_3(\Lambda)$. It follows by symmetry considerations that the zeros of this equation are the values of $\wp_{\Lambda}$ at the half-periods
$\frac{a}{2}$, $\frac{b}{2}$, $\frac{a+b}{2}$.
Indeed, since the doubly periodic function $\wp'_{\Lambda}$ is odd and has no pole at its half-periods, it vanishes at its half-periods. Since  $\wp'_{\Lambda}$
is of order $3$, these are all its zeros in a period-parallelogram. An argument using that the order of  $\wp_{\Lambda}$ is $2$, shows that  all roots of $\wp_{\Lambda}$  of are distinct.
The differential equation for $\wp_{\Lambda}$ becomes
\begin{equation}\label{eqEl29b}
(\wp'_{\Lambda})^2 (\zeta) = 4 (\wp_{\Lambda} (\zeta) - e_1
(\Lambda))(\wp_{\Lambda} (\zeta) - e_2 (\Lambda)) (\wp_{\Lambda} (\zeta) - e_3
(\Lambda)) \,,
\end{equation}
where
\begin{equation}\label{eqEl32}
e_1 (\Lambda) = \wp_{\Lambda} \left(\frac{a}{2} \right) \, ,
\quad e_2 (\Lambda) = \wp_{\Lambda} \left(\frac b2 \right) \, , \quad
e_3 (\Lambda) = \wp_{\Lambda} \left(\frac{a+b}2 \right)
\end{equation}
are the values of $\wp_{\Lambda}$ at points which are contained in the
lattice $\frac a2 \, {\mathbb Z} + \frac b2 \, {\mathbb Z}$ but are
not contained in the lattice $a{\mathbb Z} + b\, {\mathbb Z}$.

The half-periods can be found from the elliptic invariants.
The integral
\begin{equation}\label{eqEl100}
\int_y^{\infty} (4t^3-g_2 t -g_3)^{-\frac{1}{2}}dt
\end{equation}
with constants $g_2=g_2(\Lambda)$ and $g_3=g_3(\Lambda)$
defines a holomorphic function in $y$ on any simply connected domain in the complex plane which does not contain a zero of the function $t\to 4t^3-g_2 t -g_3$.
The derivative in $z$ of the function defined by the integral \eqref {eqEl100} with $y=\wp_{\Lambda}(z)$ equals
\begin{align}\label{eqEl101}
\frac{\wp'_{\Lambda}(z)}{\big(4\wp_{\Lambda}(z)^3-g_2({\Lambda})\wp_{\Lambda}(z)-g_3({\Lambda})\big)^{\frac12}},
\end{align}
which equals either $1$ or $-1$ on each suitable domain for $z$.
This shows that (after multiplying by $\pm 1$) the integral provides an inverse of $\wp_{\Lambda}$. Knowing $e_1,e_2,$ and $e_3$ we obtain the half-periods.

Put for instance $e_1=1,\;e_2=-1$ and $e_3=0$. Then $g_2=4$, $g_3=0$. The integral $\int_y^{\infty} (4t^3-4 t )^{-\frac{1}{2}}dt$ defines a holomorphic function on $\mathbb{C}\setminus \big(i[0,\infty)\cup [-1,1]\big)$. Its continuous extension $\omega_1$ to $1$ is positive, and its continuous extension $\omega_2$ to $-1$ equals $i\omega_1$.
Hence, the Weierstrass $\wp$-function corresponding to the periods $2\omega_1$ and $2\omega_2=i2\omega_1$ takes the values $1$, $-1$ and $0$ at the points $\omega_1$, $\omega_2$, $\omega_1+\omega_2$, respectively.

The lattice $\Lambda_{\tau}\stackrel{def}={\mathbb Z} + \tau \,{\mathbb Z}$
often plays a special role. If the pair $1$ and $\tau$ generates the lattice, then also the pair $1$ and $-\tau$ generates it. Hence, we may assume that ${\rm{Im}}{\tau} >0$. We will write $\wp_{\tau}\stackrel{def}=
\wp_{\Lambda_{\tau}}$, \index{$\wp_{\tau}$}
\begin{equation}\label{eqEl30a}
\wp_{\tau} (\zeta) = \frac{1}{\zeta^2} + \sum_{(n,m)
\in {\mathbb Z}^2 \atop (n,m) \ne (0,0)} \left( \frac{1}{(\zeta -\, n -
\,m \, \tau)^2} - \frac{1}{(\,n+ \tau\,m )^2} \right) , \quad \zeta \in
{\mathbb C} \setminus \Lambda_{\tau}\,.
\end{equation}
The function $\wp_{\tau}$ satisfies the differential equation
\begin{equation}\label{eqEl31}
(\wp'_{\tau})^2 (\zeta) = 4 (\wp_{\tau} (\zeta) - e_1
(\tau))(\wp_{\tau} (\zeta) - e_2 (\tau)) (\wp_{\tau} (\zeta) - e_3
(\tau)) \, ,
\end{equation}
where
\begin{equation}\label{eqEl32a}
e_1 (\tau) = \wp_{\tau} \left(\frac12 \right) \, ,
\quad e_2 (\tau) = \wp_{\tau} \left(\frac\tau2 \right) \, , \quad
e_3 (\tau) = \wp_{\tau} \left(\frac{1+\tau}2 \right)
\end{equation}
are the values of $\wp_{\tau}$ at points which are contained in the
lattice $\frac12 \, {\mathbb Z} + \frac\tau2 \, {\mathbb Z}$ but are
not contained in the lattice ${\mathbb Z} + \tau \, {\mathbb Z}$.

An arbitrary lattice $\Lambda$ can be written as $\Lambda=\alpha(\mathbb{Z} +\tau \mathbb{Z})$. By \eqref{eqEl30} the equality $ \wp_{\Lambda}(\zeta)= \alpha^{-2} \, \wp_{\tau}\left(\frac\zeta\alpha\right)$ holds, hence
\begin{equation}\label{eqEl29x}
\{e_1 (\Lambda),e_2 (\Lambda),e_3 (\Lambda)\}= \{\alpha^{-2}e_1(\tau),\alpha^{-2}e_2(\tau),\alpha^{-2}e_3(\tau)\}\,.
\end{equation}

\noindent {\bf The double branched covering defined by the Weierstra\ss  $\wp$-function.}
Let $\Lambda$ be an arbitrary lattice.
Put $z (\zeta)= \wp_{\Lambda}(\zeta)= \alpha^{-2} \, \wp_{\tau}
\left(\frac\zeta\alpha\right)$ and
$w(\zeta) = \wp'_{\Lambda}(\zeta) =\alpha^{-3} \, \wp'_{\tau}
\left(\frac\zeta\alpha\right)$. The mapping
\begin{equation}\label{eqEl29}
{\mathbb C} \backslash \Lambda \ni \zeta \to \big(z(\zeta),w(\zeta)\big)\,=\, \big(
 \wp_{\Lambda} ( \zeta ) ,
 \wp'_{\Lambda}({\zeta})
\big) \in {\mathbb C}^2 \,
\end{equation}
descends to a conformal mapping $\omega_{\Lambda}$ from the punctured torus
${\mathbb C} \backslash \Lambda \diagup \Lambda$ onto the complex hypersurface
$\mathring{{Y}}({\Lambda})$ of ${\mathbb C}^2$,
\begin{align}\label{eqEl34}
\omega_{\Lambda}:(\mathbb{C}\setminus \Lambda)\diagup \Lambda \to       \mathring{{Y}}({\Lambda})\stackrel{def} =
\Big\{(z,w)\in \mathbb{C}^2: w^2  = 4 \prod_{j=1}^3\,(z -  \, e_j (\Lambda))
\Big\}\,.
\end{align}
\index{$\omega_{\Lambda}$}
To see this we notice first that the mapping $\omega_{\Lambda}$
is one-to-one. Indeed, the mapping $\wp_{\Lambda}$ descends to a holomorphic mapping ${\sf{p}}_{\Lambda}:(\mathbb{C}\setminus \Lambda)\diagup\Lambda \to \mathbb{C}$, such that
\begin{align}\label{eqEl34a'}
\wp_{\Lambda} = {\sf{p}}_{\Lambda}\circ p\,.
\end{align}
Here $p$ denotes the projection $p:\mathbb{C}\setminus \Lambda \to (\mathbb{C}\setminus \Lambda)\diagup \Lambda$. The mapping ${\sf{p}}_{\Lambda}$ is $2$ to $1$. Further,
$\wp_{\Lambda}(\zeta)=\wp_{\Lambda}(-\zeta)$.
Hence, the preimage under $\wp_{\Lambda}$ of each point in $\mathbb{C}\setminus \{0\}$
equals $(\{\zeta\}+\Lambda)\cup (\{-\zeta\}+\Lambda)$ for some $\zeta \in \mathbb{C}$.
Since,
$\wp_{\Lambda}'(\zeta)=-\wp_{\Lambda}'(-\zeta)$, the mapping $\omega_{\Lambda}$ is one-to-one.
\index{${\sf{p}}_{\Lambda}$}

Moreover, the mapping \eqref{eqEl29} is locally conformal. Indeed, if $\wp_{\Lambda}(\zeta_0)\neq e_j({\Lambda}),\, j=1,2,3,$ then $\wp_{\Lambda}'(\zeta_0)\neq 0$
and the mapping $\zeta\to \wp_{\Lambda}(\zeta)$ is conformal in a neighbourhood of $\zeta_0$. Suppose
$\wp_{\Lambda}(\zeta_0)=e_j(\Lambda)$. The differential equation \eqref{eqEl29b}  implies
that $\wp_{\Lambda}' \wp_{\Lambda}''= 2\wp_{\Lambda}'\sum_{\ell=1}^3 \prod_{k\neq \ell}  (\wp_{\Lambda}-e_k(\Lambda))$. Hence, $\wp_{\Lambda}''(\zeta_0)=2 \prod_{k\neq j}(e_j(\Lambda)- e_k(\Lambda))\neq 0$. Hence, in this case  the mapping $\zeta\to \wp_{\Lambda}'(\zeta)$ is conformal in a neighbourhood of $\zeta_0$.

The set $\mathring{{Y}}({\Lambda})$ (see \eqref{eqEl34}) is the covering space of the double branched covering
\begin{equation}\label{eqEl34a}
\mathring{{Y}}({\Lambda})\ni (z,w)\to z \in\mathbb{C}
\end{equation}
of ${\mathbb C}$ with branch locus
\begin{equation}\label{eqEl35}
\mathring{BL}(\Lambda)\stackrel{def} =\big\{e_1(\Lambda),e_2(\Lambda),e_3(\Lambda)\big\}\,.
\end{equation}
With $BL(\Lambda)\stackrel{def}=\mathring{BL}(\Lambda)\cup\{\infty\}$ the equality
$\mathring{{Y}}({\Lambda})=\mathring{{Y}}_{BL(\Lambda)}$ holds.
\index{$\mathring{BL}(\Lambda)$} \index{$BL(\Lambda)$}

For a family of lattices $\Lambda (z)$ depending holomorphically on a complex parameter $z$ the sets $\mathring{BL}(\Lambda (z))$ (see \eqref{eqEl35})
 depend holomorphically on
$z$. Indeed, the $e_j(\Lambda)$ are the values of $\wp_{\Lambda}$ at the
half-periods of the lattice.
The half-periods of the lattice depend holomorphically on $z$, and also the Weierstra\ss\ function $\wp_{\Lambda}$ depends holomorphically on $z$ (see equation \eqref{eqEl30}).

If the family of lattices is merely smooth, then the set of finite branch points in the
fiber over $x$ depends smoothly on the point $x \in X$.

The one-point compactification $Y_{BL(\Lambda)}$ of $\mathring{Y}_{BL(\Lambda)}$ (see
Section \ref{sec:9.6}) is the covering manifold of a double branched covering of the Riemann
sphere $\mathbb{P}^1$ with branch locus $BL=\mathring{BL}(\Lambda)\cup\{\infty\}$ and finite branch locus $\mathring{BL}(\Lambda)$. We associate to the
Riemann surface $\mathbb{C}\diagup \Lambda$ the conformally equivalent Riemann surface
$Y_{BL(\Lambda)}$ which we also  denote by $Y(\Lambda)$. Notice that $\omega_{\Lambda}$ extends to a conformal mapping between the closed Riemann surfaces $\mathbb{C}\diagup \Lambda$ and $Y(\Lambda)$.
\index{$Y(\Lambda)$}

\noindent {\bf Lifts of mappings to the double branched covering and the involution.}
For each lattice $\Lambda$ the mapping ${\mathbb C} \ni \zeta \to -
\zeta$ maps $\Lambda$ onto itself. This mapping descends to an
involution of ${\mathbb C} \diagup \Lambda$, i.e. to a
self-homeomorphism $\iota_{\Lambda}$ of ${\mathbb C} \diagup \Lambda$ such
that $\iota_{\Lambda}^2 = {\rm id}$.
\index{$\iota_{\Lambda}$} \index{involution}
\smallskip

Formula (\ref{eqEl30}) implies the following equality
\begin{equation}
\label{eqEl39} \left( \wp_{\Lambda}
\left(-\zeta\right) , \wp'_{\Lambda}
\left(-\zeta\right)\right) = \left(
\wp_{\Lambda} \left(\zeta\right) ,
-\wp'_{\Lambda} \left(\zeta\right)\right) \, .
\end{equation}

Conjugate the restriction $\iota_{\Lambda}\mid (\mathbb{C}\setminus \Lambda)\diagup \Lambda$ by the conformal mapping $\omega_{\Lambda}^{-1}$.
We obtain a self-homeomorphism of $\mathring{Y}({\Lambda})$ which we denote by $\iota$. The
mapping $\iota$ satisfies the equality
\begin{equation} \label{eqEl39a}
\iota(z,w)=(z,-w)\,.
\end{equation}
\index{$\iota$}
The involution $\iota$ extends to an involution of ${Y}({\Lambda})$, denoted also by $\iota$.
By \eqref{eqEl39a} the involution $\iota$ fixes the projection to
${\mathbb P}^1$ of each point of the double branched covering space $Y(\Lambda)$, and interchanges the
sheets over each point. Hence, it fixes each of the three finite branch
points, it also fixes $\infty$, and it does not fix any other point.

\smallskip

Let $\widetilde\varphi$ be any real linear self-homeomorphism of
${\mathbb C}$ that maps $\Lambda$ onto itself, and let $\varphi$ be
the induced mapping on ${\mathbb C} \diagup \Lambda$. Then $\varphi$
commutes with $\iota_{\Lambda}$. Indeed, $\widetilde\varphi (-\zeta) = -
\widetilde\varphi (\zeta)$, $\zeta \in {\mathbb C}$.

\smallskip

Each mapping class ${\mathfrak
m}$ in $\mathfrak{M}({\mathbb C} \diagup \Lambda;\,0\diagup \Lambda,\emptyset)$ can be represented by a
self-homeomorphism of ${\mathbb C} \diagup \Lambda$ which commutes
with $\iota_{\Lambda}$. Indeed, $\mathfrak{M}({\mathbb C} \diagup \Lambda;\;0\diagup \Lambda,\,\emptyset)$ contains a mapping $\varphi$ that lifts
to a real linear self-map
$\widetilde\varphi$ of ${\mathbb C}$ which maps $\Lambda$ onto
itself. This mapping $\varphi$ commutes with $\iota_{\Lambda}$.

\smallskip

As a corollary, each mapping class ${\mathfrak
m}$ in $\mathfrak{M}(Y(\Lambda);\, s^{\infty} ,\,\emptyset)\cong \mathfrak{M}({\mathbb C} \diagup \Lambda;\,0\diagup \Lambda,\emptyset)$
is a lift of a mapping class $\mathfrak{m}_1\in \mathfrak{M}(\mathbb{P}^1;\{\infty\},\mathring{E})$ for the set $\mathring{E}\stackrel{def}= \mathring{BL}(\Lambda)\subset C_3(\mathbb{C})\diagup \mathcal{S}_3$.

This can be seen as follows.
If $\varphi \in \mathfrak{m} \in  \mathfrak{M}({\mathbb C} \diagup \Lambda;\,0\diagup \Lambda,\,\emptyset)$ commutes with  $\iota_{\Lambda}$, then $\psi=\omega_{\Lambda}\circ \varphi \circ \omega_{\Lambda}^{-1}$ commutes with $\iota$ and represents the mapping class in $\mathfrak{M}(Y(\Lambda);\, s^{\infty} ,\,\emptyset)$ that corresponds to $\mathfrak{m}$.
Suppose $\psi$ represents a mapping class $\mathfrak{m}\in \mathfrak{M}(Y(\Lambda);\, s^{\infty} ,\,\emptyset)$ and commutes with $\iota$.
Denote by $\mathring{\psi}$ the restriction of $\psi$ to $\mathring{Y}$.
In coordinates $(z,w)$ on $\mathring {Y}(\Lambda)$ we write
\begin{equation}\label{eqEl19a}
\mathring{\psi}(z,w) = (\mathring{\psi}_1 (z,w) ,\mathring{ \psi}_2 (z,w)) \,.
\end{equation}
Since $\mathring{\psi}$ commutes with $\iota$ we obtain by (\ref{eqEl39a})
\begin{equation}\label{eqEl39d}
(\mathring{\psi}_1 (z,-w) ,\mathring{ \psi}_2 (z,-w)) = \mathring{\psi} \circ \iota (z,w) = \iota \circ \mathring{\psi} (z,w)=
(\mathring{\psi}_1 (z,w) , -\mathring{\psi}_2 (z,w))\,.
\end{equation}
Hence $\mathring{\psi}_1(z,w)=\mathring{\psi}_1(z,-w)$. Since each point in $(z,w)\in \mathring{Y}_{\Lambda}$ is determined by $z$ and the sign of $w$, this means that $\mathring{\psi}_1(z,w)$ depends only on the coordinate $z \in {\mathbb C}$, not on the sheet (determined by $w$). Further, if $\iota$ fixes
$(z,w)$, then $w=0$, and by  \eqref{eqEl39d} for $w=0$, $\mathring{\psi}_2(z,w)=\mathring{\psi}_2 (z,-w)= -\mathring{\psi}_2(z,w)$. In other words, if $\iota$ fixes $(z,w)$, then it also fixes $\mathring{\psi}(z,w)$. Hence, $\mathring{\psi}$
maps the set of finite branch points (the preimage of $\mathring{BL}({\Lambda})$ under the branched covering map) onto itself, and its extension $\psi$ maps the preimage $s^{\infty}$ of
$\infty$ to itself.
We saw, that
${\psi}$ induces a self-homeomorphism ${\sf Pr}(\psi)$ of ${\mathbb P}^1$ in the
class ${\mathfrak M} ({\mathbb P}^1 ; \{\infty\} , \mathring{BL}({\Lambda}))$,
${\sf Pr}(\psi) |\mathbb{C}= {\sf Pr}( \mathring{\psi}) $, ${\sf Pr} (\mathring{\psi})(z)=\mathring{\psi}_1(z,w), \, (z,w)\in \mathbb{C}^2$.
We call ${\sf Pr}(\psi)$ the projection of $\psi$. \index{${\sf Pr}(\psi)$}
The mapping class $\mathfrak{m}$ is a lift of the mapping class $\mathfrak{m}_1$ of $\psi_1$.

Notice that the provided arguments imply also the following fact. For two mapping
classes $\mathfrak{m},\mathfrak{m}'\in\mathfrak{M}(Y(\Lambda);s^{\infty}, \emptyset)$ the equality ${\sf Pr}(\mathfrak{m}\, \mathfrak{m}')={\sf Pr}(\mathfrak{m}){\sf Pr}(\mathfrak{m}')$ holds. Indeed,
take representing maps
$\psi^1$ and $\psi^2$, that commute with the involution $\iota$. Using the arguments given above, the equality ${\sf Pr}(\psi^1\circ \psi^2)={\sf Pr}(\psi^1)\circ {\sf Pr}(\psi^2)$ follows.

Let again $\mathring{E}$ be an element of $C_3(\mathbb{C})\diagup \mathcal{S}_3$ and $E=\mathring{E}\cup \{\infty\}$. Any mapping class $\mathfrak{m}_1\in \mathfrak{M}(\mathbb{P}^1;\{\infty\},\mathring{E})$ has exactly two lifts $\mathfrak{m}_{\pm} \in \mathfrak{M}(Y_E;s^\infty,\emptyset)$. Here $Y_E$ is the double branched covering of $\mathbb{P}^1$ with branch locus ${E}$, and $s^\infty$ is the branch point over $\infty$. Indeed, for any representative $\psi_1$ of $\mathfrak{m}_1$ there are exactly two self-homeomorphisms of the double branched covering that lift $\psi_1$.
They are obtained as follows. Cut $\mathbb{P}^1$ along two disjoint simple curves $\gamma_1$ and $\gamma_2$, that join disjoint pairs of points of $E$. $\psi_1$ maps the pair of curves $\gamma_1$ and $\gamma_2$ to another pair of curves $\gamma'_1$ and $\gamma'_2$ joining (maybe different) disjoint pairs of points of $E$. The double branched covering over $\mathbb{P}^1$ is obtained in two different ways, either gluing two sheets of $\mathbb{P}^1\setminus (\gamma_1\cup\gamma_2)$ crosswise together, or gluing two sheets of $\mathbb{P}^1\setminus (\gamma'_1\cup\gamma'_2)$ crosswise together. $\psi_1$ maps $\mathbb{P}^1\setminus (\gamma_1\cup\gamma_2)$ homeomorphically onto $\mathbb{P}^1\setminus (\gamma'_1\cup\gamma'_2)$. Consider the mapping that takes the first sheet of
$\mathbb{P}^1\setminus (\gamma_1\cup\gamma_2)$ onto the first sheet of $\mathbb{P}^1\setminus (\gamma'_1\cup\gamma'_2)$ and the second sheet of
$\mathbb{P}^1\setminus (\gamma_1\cup\gamma_2)$ onto the second sheet of $\mathbb{P}^1\setminus (\gamma'_1\cup\gamma'_2)$ and lifts $\psi_1\mid \mathbb{P}^1\setminus (\gamma_1\cup\gamma_2)$. This mapping extends to a self-homeomorphism $\psi$ of the double branched covering that lifts $\psi_1$.
There is exactly one more self-homeomorphism of the double branched covering that lifts
$\psi_1$. This self-homeomorphism maps the first sheet of
$\mathbb{P}^1\setminus (\gamma_1\cup\gamma_2)$ onto the second sheet of $\mathbb{P}^1\setminus (\gamma'_1\cup\gamma'_2)$ and the second sheet of
$\mathbb{P}^1\setminus (\gamma_1\cup\gamma_2)$ onto the first sheet of $\mathbb{P}^1\setminus (\gamma'_1\cup\gamma'_2)$.
In other words, there are
two lifts of $\psi_1$ to the double branched covering of ${\mathbb P}^1$ with branch locus
$E$, and they differ by involution.

We proved the following lemma.

\begin{lemm}\label{lemEl100}
Each mapping class  $\mathfrak{m}\in \mathfrak{M}({\mathbb C} \diagup \Lambda;\,0\diagup \Lambda,\emptyset)$
is the lift of a mapping class
$\mathfrak{m}_1={\sf Pr}(\mathfrak{m})\in \mathfrak{M}(\mathbb{P}^1;\{\infty\},\mathring{BL}(\Lambda))$. Vice versa,  each  $\mathfrak{m}_1\in \mathfrak{M}(\mathbb{P}^1;\{\infty\},\mathring{BL}(\Lambda))$ has two lifts $\mathfrak{m}_{\pm}\in\mathfrak{M}({\mathbb C} \diagup \Lambda;\,0\diagup \Lambda,\emptyset)$. The lifts $\mathfrak{m}_{\pm}$ differ by involution. For two mapping classes $\mathfrak{m},\,\mathfrak{m}'\in \mathfrak{M}({\mathbb C} \diagup \Lambda;\,0\diagup \Lambda,\emptyset)$ the equality ${\sf Pr}(\mathfrak{m} \mathfrak{m}')={\sf Pr}(\mathfrak{m}){\sf Pr}(\mathfrak{m}')$ holds.
\end{lemm}

\noindent {\bf Lifts of special $(0,4)$-bundles to elliptic fiber bundles with a section.}
The following proposition relates Theorem \ref{thmEl.0} to the respective Theorem \ref{thmEl.9} for
elliptic bundles that will be formulated below.
\begin{prop}\label{propEl.2} Let $X$ be a connected Riemann surface (connected oriented smooth surface, respectively) of genus $g$ with ${m}\geq 1$ holes with base point $x_0$ and curves denoted by $\gamma_j$ that represent a  standard system of generators $e_j\in \pi_1(X,x_0)$.\\
\noindent $(1)$ Each complex analytic (differentiable, respectively) family of Riemann surfaces of type $(1,1)$ over $X$ is holomorphically (smoothly, respectively) isomorphic to the canonical double branched covering $\mathfrak{F}$ of a special holomorphic (smooth, respectively) $(0,4)$-bundle $\mathfrak{F}_1$ over $X$. The monodromy of the bundle $\mathfrak{F}$ along each $\gamma_j$ is a lift of the respective monodromy of the bundle $\mathfrak{F}_1$.\\
\noindent $(2)$ Vice versa, for each special holomorphic (smooth, respectively) $(0,4)$-bundle over $X$ and each collection $\mathfrak{m}^j$ of lifts of the $2{ g} +{ m}-1 $ monodromy mapping classes $\mathfrak{m}_1^j$ of the bundle along the $\gamma_j$ there exists a double branched covering by
a complex analytic (differentiable, respectively) family of Riemann surfaces of type $(1,1)$ with collection of monodromy mapping classes equal to the $\mathfrak{m}^j$.
For each given holomorphic (smooth, respectively) special $(0,4)$-bundle over $X$ there are up to holomorphic (smooth, respectively) isomorphisms exactly
$2^{2{g}+ {m}-1}$ holomorphic (smooth, respectively) families of Riemann surfaces of type $(1,1)$ that lift the $(0,4)$-bundle.\\
\noindent $(3)$ A lift of a special $(0,4)$-bundle is reducible if and only if the special
 $(0,4)$-bundle is reducible.
\end{prop}

\noindent {\bf Proof.}
We start with the proof of the {\bf first statement} of the proposition.
Consider a complex analytic (smooth, respectively) family of Riemann surfaces of type $(1,1)$ over $X$. By Lemma \ref{lemEl1} and Proposition \ref{propEl.2a}
we may assume that the family has the form
$(\mathcal{X}_{\Lambda},\mathcal{P}_{\Lambda}, \,\mathbold{s}_{\Lambda}\,,X)\,$
for a
holomorphic (smooth, respectively) family of lattices $\Lambda(x),\, x \in X$.

\index{$\mathring{\mathbold{BL}}(\Lambda)$} \index{$\mathbold{BL}(\Lambda)$}
We consider the complex (smooth, respectively) submanifold $\mathring{\mathbold{BL}}(\Lambda)$ of $X\times \mathbb{C}$ that intersects each fiber $\{x\}\times \mathbb{C}$ along
the set $\{x\}\times \mathring{BL}(\Lambda(x))$ (see equations
\eqref{eqEl32} and \eqref{eqEl35})
and  the complex (smooth, respectively) submanifold $\mathbold{BL}(\Lambda)$ of $X\times \mathbb{P}^1$ that intersects each fiber $\,\{x\}\times \mathbb{P}^1\,$ along the set $\,\{x\}\times BL(\Lambda(x))\,$ with $BL(\Lambda(x))\,=\, \mathring{BL}(\Lambda(x))\cup \{\infty\}$.

We prove first that
the holomorphic (smooth, respectively) family
$$
(\mathcal{X}_{\Lambda},\mathcal{P}_{\Lambda}, \,\mathbold{s}_{\Lambda}\,,X)\,
$$
is holomorphically isomorphic (isomorphic as a smooth family of Riemann surfaces, respectively) to the canonical double branched covering
$$
(\mathcal{Y}_{\mathbold{BL}(\Lambda)}, \mathcal{P},\mathbold{s}^{\infty},X)
$$
(see Section \ref{sec:9.6}) of the
special $(0,4)$-bundle $(X\times \mathbb{P}^1, {\rm pr}_1, \mathbold{BL}(\Lambda),X  )$.
To see this we recall that for each $x\in X$ the mapping $\omega_{\Lambda_x}:(\mathbb{C}\setminus \Lambda_x)\diagup \Lambda_x \to \mathring{Y}_{{BL}(\Lambda_x)}$ is a surjective conformal mapping that depends holomorphically (smoothly, respectively) on $x$ (see equalities \eqref{eqEl29} and \eqref{eqEl34}).
Hence,
the mapping defines a holomorphic (smooth, respectively) homeomorphism from $\mathcal{X}_{\Lambda}\setminus \mathbold{s}_{\Lambda}$ onto $\mathring{\mathcal{Y}}_{\mathbold{BL}(\Lambda)}$, that maps the fiber over $x$ of the first bundle conformally onto the fiber over $x$ of the second bundle.

The set ${\mathcal{Y}}_{\mathbold{BL}(\Lambda)}$
is obtained as in Section \ref{sec:9.6}
by ''adding a point to each fiber''.
As in Section \ref{sec:9.6} the mapping extends to a holomorphic (smooth, respectively) homeomorphism $\mathcal{X}_{\Lambda} \to  {\mathcal{Y}}_{\mathbold{BL}(\Lambda)}$,
that maps the fiber over $x$ conformally onto the fiber over $x$. The extension maps $\mathbold{s}_{\Lambda}$ to $\mathbold{s}^{\infty}$.
We proved that the bundles $(\mathcal{X}_{\Lambda},\mathcal{P}_{\Lambda}, \,\mathbold{s}_{\Lambda}\,,X)\,$ and $(\mathcal{Y}_{\mathbold{BL}(\Lambda)}, \mathcal{P},\mathbold{s}^{\infty},X)$ are holomorphically isomorphic (isomorphic as smooth families of Riemann surfaces, respectively).
Hence, each holomorphic (smooth, respectively) $(1,1)$-family over a finite open Riemann surface is holomorphically isomorphic (smoothly isomorphic), and in particular, isotopic to the canonical double branched covering of a holomorphic (smooth, respectively) $(0,4)$-bundle.

We identify the mapping class groups in the fiber over the base point $x_0$ of the bundles
$(\mathcal{X}_{\Lambda},\mathcal{P}_{\Lambda}, \,\mathbold{s}_{\Lambda}\,,X)\,$ and $(\mathcal{Y}_{\mathbold{BL}(\Lambda)}, \mathcal{P},\mathbold{s}^{\infty},X)$ using the extension of the isomorphism $\omega_{\Lambda_0}$ to the closed fiber $\mathbb{C}\diagup {\Lambda_0}$.
Having in mind this identification, we  prove  now  that  the  monodromy  mapping  class along each $\gamma_j$ of  the  $(1,1)$-bundle
$(\mathcal{X}_{\Lambda},\mathcal{P}_{\Lambda}, \,\mathbold{s}_{\Lambda}\,,X)\,$
is a lift of the monodromy mapping class along $\gamma_j$ of the special $(0,4)$-bundle $(X\times \mathbb{P}^1, {\rm pr}_1, \mathbold{BL}(\Lambda),X  )$. Parameterise $\gamma_j$ by the unit interval $[0,1]$.
Write $\Lambda_t\stackrel{def}=\Lambda(\gamma_j(t))=a(t)\mathbb{Z}+b(t)\mathbb{Z},\, t\in[0,1],$ for  smooth functions $a$ and $b$ on $[0,1]$.
For $t \in [0,1]$ we denote by $\widetilde{\varphi}^t$ the real linear self-homeomorphism of $\mathbb{C}$ that maps $a(0)$ to $a(t)$ and $b(0)$ to $b(t)$. Let $\varphi^t$ be the homeomorphism from the fiber over $\gamma_j(0)$ onto the fiber over $\gamma_j(t)$ of the $(1,1)$-bundle that lifts to $\widetilde{\varphi}^t$. Then ${\varphi}^0= {\rm Id}$ and $({\varphi}^1)^{-1}$ represents the monodromy mapping class of the $(1,1)$-bundle $(\mathcal{X}_{\Lambda},\mathcal{P}_{\Lambda}, \,\mathbold{s}_{\Lambda}\,,X)\,$
along $\gamma_j$. Let $\mathring{\psi}^t:\mathring{Y}(\Lambda_0) \to \mathring{Y}(\Lambda_t)$ be
obtained from the commutative diagram
$$
\xymatrix{(\mathbb{C}\setminus \Lambda_0)\diagup \Lambda_0
 \ar[rr]^{\mathring{\varphi}^t}  \ar[d]^{\omega_{\Lambda_0}} &&(\mathbb{C}\setminus \Lambda_t)\diagup \Lambda_t \ar[d]^{\omega_{\Lambda_t}}\\
\ar[rr]^{\mathring{\psi}^t} {\mathring{Y}(\Lambda_0)}  && \mathring{Y}(\Lambda_t)}
$$
where $\mathring{\varphi}^t=\varphi^t|(\mathbb{C}\setminus \Lambda_0)\diagup \Lambda_0$.

We use the restriction to $\mathring{Y}(\Lambda_0)$ of the coordinates $(z,w)$ on $\mathbb{C}^2$. Let $\iota_0$ be the involution on $\mathring{Y}(\Lambda_0)$ and $\iota_t$ the involution on $\mathring{Y}(\Lambda_t)$. Choose respective coordinates on
$\mathring{Y}(\Lambda_t)$ and write $\mathring{\psi}^t(z,w)=(\mathring{\psi}^t_1(z,w), \mathring{\psi}^t_2(z,w)), \; (z,w) \in \mathring{Y}(\Lambda_0)$. The equality $\widetilde{\varphi}^t(-\zeta)\,=\,-\widetilde{\varphi}^t(\zeta),\; \zeta \in\mathbb{C},\,$ can be written as
 $\mathring{\varphi}^t \,\iota_{\Lambda_0}= \iota_{\Lambda_t}\, \mathring{\varphi}^t$. From the diagram we obtain
with $\iota_t= \omega_{\Lambda_t} \iota_{\Lambda_t} (\omega_{\Lambda_t})^{-1}$,  $\iota_0= \omega_{\Lambda_0} \iota_{\Lambda_0} (\omega_{\Lambda_0})^{-1}$
\begin{equation}\label{eqEl19d}
\mathring{\psi}^t\,\iota_0= \iota_t \, \mathring{\psi}^t\,.
\end{equation}
In coordinates $(z,w)$ this means 
$(\mathring{\psi}^t_1(z,-w), \mathring{\psi}^t_2(z,-w))=
(\mathring{\psi}^t_1(z,w), -\mathring{\psi}^t_2(z,w))$. Hence, $\mathring{\psi}^t_1(z,w)=\mathring{\psi}^t_1(z,-w)$, so that $\mathring{\psi}^t_1$ depends only on the coordinate $z \in \mathbb{C}$.
Moreover, if $\iota_0$ fixes $(z,w)$ (equivalently, if $z\in \mathring{BL}(\Lambda_0)$ ), then $w=0$, hence $\mathring{\psi}^t_2(z,w)=\mathring{\psi}^t_2(z,-w)=-\mathring{\psi}^t_2(z,w)$, i.e. $\iota_t$ fixes $\mathring{\psi}^t(z,w)$ (equivalently, $\mathring{\psi}^t_1(z,w)\in \mathring{BL}(\Lambda_t)$).
We saw that
the self-homeomorphism $\mathring{\psi}^t_1$ of $\mathbb{C}$ maps the set  $\mathring{BL}(\Lambda_0)$ onto the set $\mathring{BL}(\Lambda_t)$.
Each $\mathring{\psi}^t_1$ extends to a self-homeomorphism $\psi^t_1$ of $\mathbb{P}^1$ that maps the set of distinguished points $BL(\Lambda_0)\stackrel{def}= \mathring{BL}(\Lambda_0)\cup\{\infty\}$ onto the set of distinguished points $BL(\Lambda_t)\stackrel{def}= \mathring{BL}(\Lambda_t)\cup\{\infty\}$. Hence,  $\psi^1_1$ represents the monodromy mapping class along $\gamma_j$ of the special $(0,4)$-bundle with set of finite distinguished points $\mathring{BL}(\Lambda_t)$ in the fiber over $\gamma_j(t)$.
The mapping $\mathring{\psi}^1=(\mathring{\psi}_1^1,\mathring{\psi}_2^1)$ is equal to $\omega_{\Lambda_0} \circ \mathring{\varphi}^1 \circ \omega_{\Lambda_0}^{-1}$. Identifying the mapping class groups of $\mathbb{C}\diagup \Lambda_0$ with the mapping class group of $Y_{BL(\Lambda_0)}$ by the isomorphism induced by $\omega_{\Lambda_0}$, we may identify the monodromy of the bundle
$(\mathcal{X}_{\Lambda},\mathcal{P}_{\Lambda}, \,\mathbold{s}_{\Lambda}\,,X)\,$ along $\gamma_j$ with the mapping class of the extension $\psi^1$ of
$\mathring{\psi}^1=(\mathring{\psi}_1^1,\mathring{\psi}_2^1)$. The
mapping $\psi_1^1$ is the projection of
$\psi^1$ under the double branched covering from $\mathbb{C}\diagup\Lambda_0$ with distinguished point $0\diagup\Lambda_0$ onto $\mathbb{P}^1$ with distinguished points ${BL}(\Lambda_0)\cup\{\infty\}$. We proved that the monodromy mapping class of the bundle $(\mathcal{X}_{\Lambda}, \mathcal{P}_{\Lambda}, \mathbold{s}_{\Lambda},X)$ along each $\gamma_j$ is a lift of the respective monodromy mapping class of the bundle $(X\times \mathbb{P}^1, {\rm pr}_1, \mathbold{BL}(\Lambda),X)$. The first statement is proved.

We prove now the {\bf second statement}.
Consider the special holomorphic, or smooth, respectively, $(0,4)$-bundle $\mathfrak{F}_1\stackrel{def}=\big(X\times \mathbb{P}^1, {\rm pr}_1, \mathbold{E}, X\big)$. The complex or smooth, respectively, submanifold $\mathbold{E}$ of $X\times \mathbb{P}^1$  intersects each fiber $\{x\}\times \mathbb{P}^1$ along the set $\{x\}\times E_x$, with  $E_x= (\mathring{E}_x\cup\{\infty\})$, where $\mathring{E}_x\subset C_3(\mathbb{C})\diagup \mathcal{S}_3$.
In Section \ref{sec:9.6} we obtained the canonical lift $(\mathcal{\mathcal{Y}}_{\mathbold{E}},\mathcal{P},\mathbold{s}^{\infty},X)$. Denote the monodromy mapping class along $\gamma_j$ of the canonical lift
by $\mathfrak{m}^j_+$.
Recall that each mapping class $\mathfrak{m}_1^j$ has two lifts and they differ by involution.
Let $\mathfrak{m}^j_-$ be the mapping class whose representatives differ from those of $\mathfrak{m}^j_+$ by composition with the involution $\iota$ of the fiber over the base point.

Take any subset $J$ of $\{1,2,\ldots,2{g}+{m}-1 \}$. The lift of the $(0,4)$-bundle whose monodromy mapping class along $\gamma_j$ equals $\mathfrak{m}^j_+$ if $j\notin J$, and
$\mathfrak{m}^j_-$ if $j\in J$, is obtained as follows.
Let $\tilde U$ be the subset of the universal covering
$\tilde X  \overset{{ \sf{P}}}{-\!\!\!\longrightarrow} X$ that was used in Section \ref{sec:9.2} in the proof of Theorem \ref{thmEl1}, and let $D$, and $\tilde{D}_j,\,j=0,\ldots,$ be as in Section \ref{sec:9.2}. Consider the lift of the bundle $(\mathcal{\mathcal{Y}}_{\mathbold{E}},\mathcal{P},\mathbold{s}^{\infty},X)$
to $\tilde X$ and restrict the lifted bundle to $\tilde U$. Denote the obtained bundle on $\tilde U$ by $\tilde{\mathfrak{F}}$.
For each point $x\in D$ we let $\tilde{x}_j\in \tilde{D}_j,\, j=0,\ldots,$ be the points
with ${\sf{P}}(\tilde{x}_j)=x$.
We identify the fiber of $\tilde{\mathfrak{F}}$ over $\tilde{x}_j,\,j=0,\ldots,$ with the fiber ${Y}_x$, $x={\sf{P}}(\tilde{x}_j)$, of the canonical double branched covering of the
$(0,4)$-bundle. $\mathring{Y}_x$ is obtained from $Y_x$ by removing the point of $\mathbold{s}^{\infty}$ from $Y_x$.

Glue for each $x\in D$ and each $j=1,\ldots,$ the fiber of $\tilde{\mathfrak{F}}$ over $\tilde{x}_j$ to the fiber of $\tilde{\mathfrak{F}}$ over $\tilde{x}_0$ using the identity if $j\notin J$ and the mapping $\iota$
if $j\in J$.
More detailed, the gluing mapping of the punctured fibers in the case $j\in J$ equals
\begin{equation}\label{eqEl19e}
(\tilde{x}_j,z,w)\to (\tilde{x}_0,z,-w),\;\; x\in D,\; (z,w)\in \mathring{Y}_x\,.
\end{equation}
Since the gluing mappings are holomorphic (smooth, respectively), we obtain a complex analytic (smooth, respectively) family over $X$ of double branched coverings of $\mathbb{C}$.
Extend the family to a complex analytic (smooth, respectively) family of double branched coverings of $\mathbb{P}^1$ over $X$. This can be done as in Section \ref{sec:9.6} by considering
the families $\mathcal{Y}^{r_n}_{\mathbold{E}|{X}_n}$ for an exhausting sequence of relatively compact open subsets $X_n$ of $X$ and a suitable sequence of real numbers $r_n$.
We obtained
for each collection of lifts of the  $\mathfrak{m}^j_1$ a complex analytic (smooth, respectively) family of Riemann surfaces of type $(1,1)$ with given fiber over the base point and with monodromy mapping classes equal to this collection. It is easy to see that
two holomorphic (smooth, respectively) $(1,1)$-bundles that lift a given $(0,4)$-bundle over a finite open Riemann surface (smooth oriented surface, respectively) are holomorphically isomorphic
(smoothly isomorphic, respectively), if and only if their monodromies coincide.
Hence, for each given holomorphic (smooth, respectively) special $(0,4)$-bundle over $X$ there are up to holomorphic (smooth, respectively) isomorphisms exactly
$2^{2{g}+ {m}-1}$ holomorphic (smooth, respectively) families of Riemann surfaces of type $(1,1)$ that lift the $(0,4)$-bundle. The second statement is proved.

It remains to prove the {\bf third statement}. Let
$$
{\sf Pr}:(\mathcal{X},\mathcal{P},\mathbold{s},X)\to (X\times \mathbb{P}^1,{\rm pr}_1,\mathbold{E},X)
$$
be a double branched covering of a $(0,4)$-bundle by a $(1,1)$-family of Riemann surfaces.
If the special $(0,4)$-bundle is reducible, then there is a simple closed curve $\gamma$ that
divides the fiber $\mathbb{P}^1$ over the base point $x_0$ into two connected components, each of which contains two
distinguished points, such that the following holds. The curve $\gamma$ is mapped by each monodromy mapping class of the $(0,4)$-bundle to a curve that is free homotopic to
$\gamma$.
The preimage of the closed curve $\gamma$ under the covering map consists of two simple closed curves
$\tilde{\gamma}_1$ and $\tilde{\gamma}_2$, that are homotopic to each other in $\mathcal{P}^{-1}(x)\setminus\mathbold{s}$, and each of
the two curves cuts the torus into an annulus. Each monodromy mapping class of the $(1,1)$-bundle, being a lift
of the respective monodromy mapping class of the $(0,4)$-bundle, takes
$\tilde{\gamma}_1$ to a curve that is free homotopic to
$\tilde{\gamma}_1$ (and also to $\tilde{\gamma}_2$). Hence, the $(1,1)$-bundle is also reducible.

Suppose now that a double branched covering $\mathfrak{F}=(\mathcal{X},\mathcal{P},\mathbold{s},X)$ of
a special $(0,4)$-bundle ${\sf Pr}(\mathfrak{F})=(X\times \mathbb{P}^1, {\rm pr}_1, \mathbold{E},X)$
is reducible and prove that the $(0,4)$-bundle is reducible. Let
$x_0\in X$ be the base point of $X$, and let the fiber of the bundle $\mathfrak{F}$ over $x_0$
be $\mathbb{C}\diagup\Lambda$ with distinguished point $0\diagup\Lambda$.
Any admissible system of curves on the torus $\mathbb{C}\diagup\Lambda$ with a distinguished point contains exactly one curve.

Let $\gamma_0$ be a simple closed curve in $(\mathbb{C}\setminus \Lambda)\diagup\Lambda$  that reduces all monodromy mapping classes of the bundle $\mathfrak{F}$.
The curve  $\gamma_0$ is free isotopic in $\mathbb{C}\diagup\Lambda$ to a curve $\gamma$ with base point $0\diagup \Lambda$ that lifts under the covering map ${\sf P} :\mathbb{C}\to \mathbb{C}\diagup\Lambda$ to the straight line segment that joins $0$ with another lattice point $\lambda$ in $\Lambda$. Since $\gamma_0$ is free isotopic in $\mathbb{C}\diagup\Lambda$ to
a simple closed curve on the torus, $\lambda$ is a primitive element of the lattice. Indeed, with $\Lambda=\{na+mb:n,m\in \mathbb{Z}\}$ for real linearly independent complex numbers $a$ and $b$ we get $\lambda=n(\lambda)a+m(\lambda)b$, where $n(\lambda)$ and $m(\lambda)$ are relatively prime integer numbers. Then there is another element $\lambda'=n(\lambda')a+m(\lambda')b\in \Lambda$ such that
\[
\begin{vmatrix}
n(\lambda) & n(\lambda') \\
m(\lambda) & m(\lambda') \\
\end{vmatrix}
= 1\,.
\]
\noindent The two elements $\lambda$ and $\lambda'$ generate $\Lambda$. Multiplying $\Lambda$ by a non-zero complex number and changing perhaps $\lambda'$ to $-\lambda'$, we may assume that $\lambda=1$ and $\lambda'=\tau$ with ${\rm Im}\tau>0$.
After similar changes of all fibers we obtain an isomorphic bundle
for which the fiber
over $x_0$ equals $\mathbb{C}\diagup\Lambda_{\tau}$. After a free isotopy in $(\mathbb{C}\setminus\Lambda_{\tau}) \diagup\Lambda_{\tau}$ we may assume that the reducing curve $\gamma_0$ lifts under the covering map $ {\sf P}:\mathbb{C}\to \mathbb{C}\diagup \Lambda_{\tau}$ to the segment $\frac{1+\tau}{2}+[0,1]\subset \mathbb{C}$.
Denote by $\gamma$ the curve in the closed torus $\mathbb{C}\diagup \Lambda_{\tau}$ that lifts to $[0,1]$,
and by $\gamma'$ the curve on the closed torus that lifts to $[0,\tau]$. $\gamma$ and $\gamma'$ represent a pair of generators of the fundamental group of the closed torus with base point $0\diagup \Lambda_{\tau}$.

The complement $(\mathbb{C}\diagup\Lambda_{\tau}) \setminus \gamma_0$ of $\gamma_0$ in the fiber over $x_0$ is a topological annulus with distinguished point $0\diagup  \Lambda_{\tau}$. Since $\gamma_0$ reduces all monodromy mapping classes of the $(1,1)$-family $\mathfrak{F}$, each monodromy mapping class has a representative that fixes $\gamma_0$ pointwise. These representatives map the topological annulus $(\mathbb{C}\diagup\Lambda_{\tau}) \setminus \gamma_0$ homeomorphically onto itself, fixing the distinguished point and fixing the boundary pointwise, or fixing the boundary pointwise after an involution. Hence, each monodromy is a power of a Dehn twist about $\gamma_0$, maybe, composed with an involution.

The
Dehn twist on $\;{\mathbb C} \diagup \Lambda_{\tau}\;$  about $\,\gamma_0\,$ can be represented by a
self-homeomorphism $\psi_{\tau}$ of ${\mathbb C} \diagup \Lambda_{\tau}$
which lifts under ${\sf P}$ to the real linear self-map $\widetilde{\psi}_{\tau}$ of
${\mathbb C}$ which maps $1$ to $1$ and $\tau$ to $1+\tau$. This can
be seen by looking at the action of $\widetilde\psi_{\tau}$ on the lifts of the
curves $\gamma$ and $\gamma'$. $\widetilde\psi_{\tau}$ takes $\widetilde{\gamma}=[0,1]$ to itself, and $\widetilde{\gamma}'=[0,\tau]$ to a curve that
is isotopic to $[0,\tau+1]$ through simple closed arcs in $\mathbb{C}$ with fixed endpoints, the interiors of the arcs avoiding $\Lambda_{\tau}$.
Note that the
real $2 \times 2$ matrix corresponding to $\widetilde\psi_{\tau}$ is
$\begin{pmatrix} 1 &({\rm Im} \, \tau)^{-1} \\ 0 & 1 \end{pmatrix}$.

We consider the double branched covering ${\sf{p}}_{\Lambda_{\tau}}: ({\mathbb C}\setminus \Lambda_{\tau})
\diagup \Lambda_{\tau}\to \mathbb{C}$ (see \eqref{eqEl34a'}) with branch locus $\mathring{BL}(\Lambda_{\tau})=\{e_1(\tau), e_2(\tau), e_3(\tau)\}$ determined by the Weierstra\ss\ $\wp$-function,
$\wp_{\Lambda_{\tau}}=      {\sf{p}}_{\Lambda_{\tau}} \circ p$ (with $p$ being the projection $p: \mathbb{C}\setminus \Lambda_{\tau} \to (\mathbb{C}\setminus \Lambda_{\tau})\diagup \Lambda_{\tau}$). Consider the continuous extension ${\sf{p}}_{\Lambda_{\tau}}^c: {\mathbb C} \diagup \Lambda_{\tau}\to \mathbb{P}^1$ of ${\sf{p}}_{\Lambda_{\tau}}$. The Weierstrass $\wp$-function extends to a mapping, also denoted by $\wp_{\Lambda_{\tau}}$,
$\wp_{\Lambda_{\tau}}: \mathbb{C}\to \mathbb{P}^1$, that takes $0$ to $\infty$, $\frac{1}{2}$ to $e_1$, $\frac{\tau}{2}$ to $e_2$, and $\frac{1+\tau}{2}$ to $e_3$.
The line segments $[0,\frac{1}{2}]$, $[0,\frac{\tau}{2}]$,
$[\frac{1}{2}, \frac{1+\tau}{2}]$,
and $[\frac{\tau}{2}, \frac{1+\tau}{2}]$
are mapped under ${\wp}_{\Lambda_{\tau}}$  to simple curves $\gamma_{\infty ,e_1}$, $\gamma_{\infty ,e_2}$, $\gamma_{e_1,e_3}$, and
$\gamma_{e_2,e_3}$
in $\mathbb{P}^1$, each of which joins the first-mentioned point with the last-mentioned point. (See Figure \ref{figwp}.) The union of the four curves (each with suitable orientation) is a closed curve in $\mathbb{P}^1$ (the union of the real axis with the point $\infty$ on Figure \ref{figwp}) that divides $\mathbb{P}^1$ into two connected components $\mathcal{C}_1$ and $\mathcal{C}_2$. The Weierstrass $\wp$-function ${\wp}_{\Lambda_{\tau}}$ maps the open parallelogram $R$ in the complex plane with vertices $0$, $1$, $\frac{\tau}{2} $, $1+\frac{\tau}{2}$ conformally onto its image. Indeed, the "double" of this parallelogram, i.e. the parallelogram with vertices $0$, $1$, $\tau $, $\tau +1$, is the interior of a fundamental polygon for the covering
$\mathbb{C}\to \mathbb{C}\diagup {\Lambda_{\tau}}$ and the equality ${\wp}_{\Lambda_{\tau}}(\zeta)= {\wp}_{\Lambda_{\tau}}(-\zeta + n + m \tau ) ,\, n,m\in \mathbb{Z}  $ holds.
By the last equation we get in particular
\begin{equation}\label{eqEl19d'}
{\wp}_{\Lambda_{\tau}}(\frac{\tau +1}{2}+t)={\wp}_{\Lambda_{\tau}}(\frac{\tau +1}{2}-t)\;\;,t\in [0,\frac{1}{2}]\,.
\end{equation}
After perhaps relabeling $\mathcal{C}_1$ and $\mathcal{C}_2$, we may assume that ${\wp}_{\Lambda_{\tau}}$ maps the open parallelogram $R_-$ with vertices $0,\,\frac{1}{2},\frac{\tau}{2},\frac{1+\tau}{2},$ (the ''left half'' of $R$) conformally onto $\mathcal{C}_1$ (the lower half-plane (shadowed) in the case of Figure \ref{figwp}), and then it maps the parallelogram $R_+$ with vertices $\frac{1}{2},\,1,\, \frac{1+\tau}{2},1+\frac{\tau}{2},$ (the ''right half'' of $R$) conformally onto $\mathcal{C}_2$ (the upper half-plane in the case of Figure \ref{figwp}).

\begin{figure}[h]
\begin{center}
\includegraphics[width=8cm]{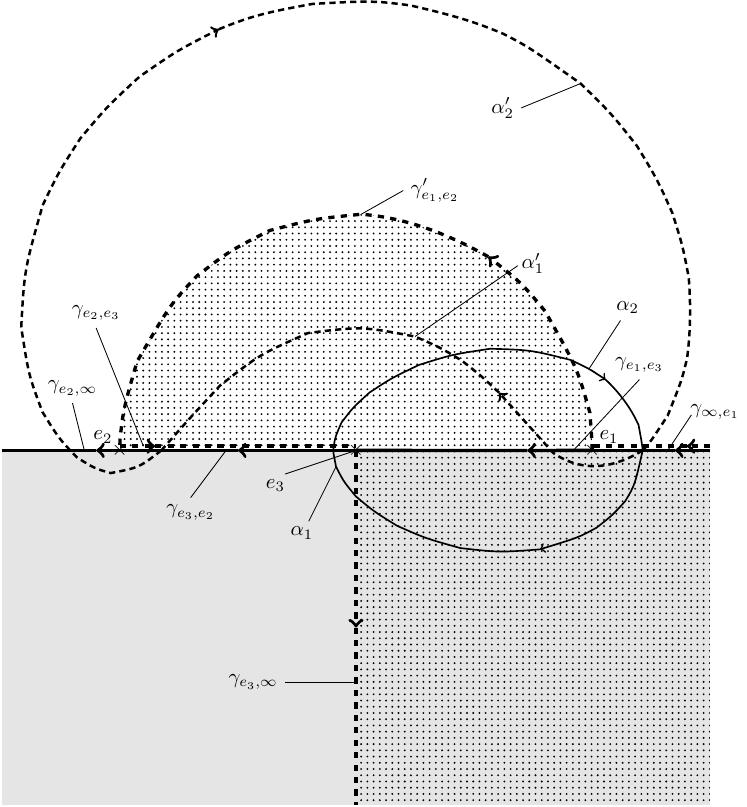}
\end{center}
\caption{The semi-conjugation of a Dehn twist to a half-twist.\\ In the figure we consider the torus $\mathbb{C}\diagup \Lambda$ for $\Lambda=
2\alpha(\mathbb{Z}+i\mathbb{Z})$ with a real number $\alpha$\\ such that $e_1=\wp_{\Lambda}(\alpha)=1$, $e_2=\wp_{\Lambda}(i \alpha)=-1$, $e_3=\wp_{\Lambda}((1+i)\alpha)=0$.}
\label{figwp}
\end{figure}

The mapping $\tilde{\psi}_{\tau}$ fixes the segment $[0,\frac{1}{2}]$, and translates the segment $[\frac{\tau}{2}, \frac{1+\tau}{2}]$ by $\frac{1}{2}$. It maps $R_-$ to the parallelogram $\tilde{\psi}_{\tau}(R_-) \subset R$ whose horizontal sides are $[0,\frac{1}{2}]$ and $[\frac{1+\tau}{2}, \frac{2+\tau}{2}]$.

The Weierstraß $\wp$-function ${\wp}_{\Lambda_{\tau}}$ maps
$[\frac{1}{2}+\frac{\tau}{2}, 1+ \frac{\tau}{2}]$
to $\gamma_{e_3,e_2}$ which equals $\gamma_{e_2,e_3}$ with inverted
orientation (see \eqref{eqEl19d'}), and takes the segment $[0,\frac{1}{2}]$ to a curve  $\gamma_{\infty, e_1}$ that joins $\infty$ with $e_1$.
It maps the segment $[0,\frac{1+\tau}{2}]$ to a curve $\gamma_{\infty, e_3}$ which is contained in $\mathcal{C}_1$ except of its endpoints and joins $\infty$ and $e_3$. The
segment $[\frac{1}{2},1+\frac{\tau}{2}]$ is mapped to a curve $\gamma'_{e_1,e_2}$ with initial point $e_1$ and terminal point $e_2$,
which is contained in  $\mathcal{C}_2$ except its endpoints, so that
the union of the curves $\gamma_{\infty,e_1}$, $\gamma'_{e_1,e_2}$,
$\gamma_{e_2,e_3}$, and $\gamma_{e_3,\infty}$
is the oriented boundary of a domain $\mathcal{C}'_1$,
i.e. the domain $\mathcal{C}'_1$ is ''on the left'' when walking along the oriented curve. (In Figure \ref{figwp}  $\mathcal{C}'_1$ is the dotted domain.)  Here $\gamma_{e_3,\infty}$ equals $\gamma_{\infty,e_3}$
with inverted orientation. Hence,  ${\wp}_{\Lambda_{\tau}}$ maps $\tilde{\psi}_{\tau}(R_-)$ onto $\mathcal{C}'_1$.

The mapping ${\wp}_{\Lambda_{\tau}}\circ \tilde{\psi}_{\tau}$ takes $R_-$ to the domain $\mathcal{C}'_1$.
The mapping $\varphi_{\tau,1}\stackrel{def}={\wp}_{\Lambda_{\tau}}\circ \tilde{\psi}_{\tau} \circ ({\wp}_{\Lambda_{\tau}}|R_-)^{-1}$ takes $\mathcal{C}_1$ onto $ \mathcal{C}_1' $. It extends continuously to a homeomorphism between closures. The extension
fixes $\gamma_{\infty,e_1}$ pointwise, it takes $\gamma_{e_1,e_3}$ to $\gamma'_{e_1,e_2}$ fixing $e_1$, it takes $\gamma_{e_3,e_2}$ to $\gamma_{e_2,e_3}$ mapping $e_2$ to $e_3$
and $e_3$ to $e_2$, and
it takes $\gamma_{e_2,\infty}$ to $\gamma_{e_3,\infty}$ fixing $\infty$.
A simple arc  $\alpha_1$ in the closure $\overline{\mathcal{C}_1}$ (i.e. an arc without self-intersection)  with initial point on $\gamma_{\infty,e_1}$ and terminal point on
$\gamma_{e_3,e_2}$ is mapped by the extension of $\varphi_{\tau,1}$ to a simple arc $\alpha_1'$ in $\overline{\mathcal{C}'_1}$ with initial point on $\gamma_{\infty,e_1}$ and terminal point on $\gamma_{e_2,e_3}$.

A similar argument for the parallelogram $R$ replaced by the parallelogram with vertices $\pm\frac{1}{2}$, $\frac{\tau}{2}\pm \frac{1}{2}$ and the domain $\mathcal{C}_1$ replaced by $\mathcal{C}_2$ shows, that with $R_+ -1\stackrel{def}=\{z-1:z\in R_+\}$
the mapping $\varphi_{\tau,2}\stackrel{def}={\wp}_{\Lambda_{\tau}}\circ \tilde{\psi}_{\tau} \circ ({\wp}_{\Lambda_{\tau}}|R_+ -1)^{-1} $ takes $\mathcal{C}_2$ homeomorphically onto  $\mathcal{C}'_2$ and extends to a homeomorphism between closures. Here $\mathcal{C}'_2$
is the domain that is bounded by the oriented curve $  \gamma'_{e_2,e_1}
\cup\gamma_{e_1,\infty}  \cup  \gamma_{\infty,e_3}\cup \gamma_{e_3,e_2}$.
Moreover, by the periodicity properties of ${\wp}_{\Lambda_{\tau}}$ (see equality
\eqref{eqEl19d'})
the extensions of  $\varphi_{\tau,1}$ and  $\varphi_{\tau,2}$ coincide on the boundary $\partial \mathcal{C}_1=\partial \mathcal{C}_2$.
Similarly as before
the extension of the mapping $\varphi_{\tau,2}\stackrel{def}={\wp}_{\Lambda_{\tau}}\circ \tilde{\psi}_{\tau} \circ ({\wp}_{\Lambda_{\tau}}|R_+ -1)^{-1} $ to the closure of $\mathcal{C}_2$ takes simple arcs $\alpha_2$ in $\overline{\mathcal{C}_2}$ with initial point on $\gamma_{e_3,e_2}$  and terminal point on $\gamma_{\infty,e_1}$
to simple arcs  $\alpha_2'$ in $\overline{\mathcal{C}'_2}$ with initial point on $\gamma_{e_2,e_3}$
and terminal point on $\gamma_{\infty,e_1}$.

It follows that
the mapping $\varphi_{\tau}$ that is equal to $\varphi_{\tau,j}$ on $\mathcal{C}_j$ extends to a self-homeomorphism of $\mathbb{P}^1$, denoted again by $\varphi_{\tau}$, whose mapping class is conjugate to $\mathfrak{m}_{\sigma_1}$. Indeed,  $\varphi_{\tau}$ takes a simple closed curve $\alpha_1\cup \alpha_2$ with $\alpha_1$ and $\alpha_2$ as before to the simple closed curve $\alpha'_1\cup \alpha'_2$ and takes a simple closed curve that surrounds $\gamma_{\infty,e_1}$
and no other  point in the branch locus
to a simple closed curve that surrounds $\gamma_{\infty,e_1}$ and no other  point in the branch locus.
The mapping class,
that is represented by a half-twist around the interval $[e_2,e_3]$ (see Section  \ref{sec:2.2a}), is determined by the following properties. It takes the homotopy class of
 $\alpha_1\cup \alpha_2$ to the homotopy class of $\alpha'_1\cup \alpha'_2$,  and it takes the homotopy class of
a simple closed curve on $\mathbb{P}^1$, that surrounds $\gamma_{\infty,e_1}$
and no other  point in the branch locus, to a curve that is isotopic to this curve on $\mathbb{C}\setminus \{e_1,e_2,e_3\}$.

Since $\tilde{\psi}_{\tau}|R_{-}=( p|R)^{-1} \circ \psi_{\tau}\circ  p|R_- $
and ${\wp}_{\Lambda_{\tau}}\circ \tilde{\psi}_{\tau}|R_-= \varphi_{\tau,1}\circ{\wp}_{\Lambda_{\tau}}|R_-$ we obtain
${\wp}_{\Lambda_{\tau}}\circ (p |R)^{-1}\circ \psi_{\tau} =\varphi_{\tau,1}\circ{\wp}_{\Lambda_{\tau}}\circ (p |R_-)^{-1}  $,
hence, since ${\wp}_{\Lambda_{\tau}}={\sf p}_{\Lambda_{\tau}}\circ p$ (see \eqref{eqEl34a'}), the equality
${\sf p}_{\Lambda_{\tau}} \circ \psi_{\tau}=\varphi_{\tau,1}\circ {\sf p}_{\Lambda_{\tau}}=\varphi_{\tau}\circ {\sf p}_{\Lambda_{\tau}}$ holds on $p(R_-)$. By the same reason the equality
${\sf p}_{\Lambda_{\tau}} \circ \psi_{\tau}=\varphi_{\tau,2}\circ {\sf p}_{\Lambda_{\tau}}
=\varphi_{\tau}\circ {\sf p}_{\Lambda_{\tau}}$
holds on $p (R_+)$. Since $ p(R_-\cup R_+)$ is dense in $\mathbb{C}\diagup \Lambda_{\tau}$, the equality ${\sf p}_{\Lambda_{\tau}} \circ \psi_{\tau}=\varphi_{\tau}\circ {\sf p}_{\Lambda_{\tau}}$
holds on $\mathbb{C}\diagup \Lambda_{\tau}$. We proved that the projection of $\psi_{\tau}$ (see Section \ref{sec:9.6}) is homotopic to a conjugate of $\mathfrak{m}_{\sigma_1}$, which is a reducible mapping class. We proved that the  $(0,4)$-bundle is reducible.
Proposition \ref{propEl.2} is proved. \hfill $\Box$

\medskip

Let $\widehat{\mathfrak{F}}$ be an isomorphism class of smooth $(1,1)$-bundles over the circle. By the analog of Proposition \ref{propEl.2} for bundles over the circle  each smooth $(1,1)$-bundle $\mathfrak{F}$ over $\partial {\mathbb D}$  is isomorphic to the canonical double branched covering of a smooth special $(0,4)$-bundle $\mathfrak{F}_1={\sf Pr} (\mathfrak{F})$ over $\partial {\mathbb D}$ (the projection of $\mathfrak{F}$). Moreover, by Lemma \ref{lemEl100} the projections of the monodromy homomorphisms of isomorphic bundles over $\partial {\mathbb D}$ are conjugate, hence the projections of isomorphic bundles are isomorphic. Vice versa, isomorphic special $(0,4)$-bundles over $\partial {\mathbb D}$ lift to isomorphic  $(1,1)$-bundles.  An isomorphism class of  special $(0,4)$-bundles over $\partial {\mathbb D}$
corresponds to a conjugacy class of $3$-braids modulo center $\widehat{b\diagup \mathcal{Z}_3}$. By Corollary \ref{cor3.2}  for any braid $b$ that is associated to a bundle representing the class ${{\sf Pr}(\widehat{\mathfrak{F}})}$ the equality $\mathcal{M}(\hat{b})=\mathcal{M}(\reallywidehat{b \Delta_3^{2k}})= \mathcal{M}(\reallywidehat
{\mathfrak{m}_{b,\infty}})=\mathcal{M}(\widehat{{\sf Pr}(\mathfrak{F})})=
\mathcal{M}({{\sf Pr}(\widehat{\mathfrak{F}})})$ holds.
Proposition \ref{propEl.2} implies the following corollary.

\begin{cor}\label{corEl2}
\begin{align*}
\mathcal{M}(\widehat{\mathfrak{F}})=\mathcal{M}(\hat b)\,.
\end{align*}
\end{cor}

\noindent {\bf Proof.}
$\mathcal{M}(\widehat{\mathfrak{F}})$ is the supremum of the conformal modules of annuli on which there exists a holomorphic $(1,1)$-bundle that represents $\widehat{\mathfrak{F}}$.
$\mathcal{M}({{\sf Pr}(\widehat{\mathfrak{F}})})$ is the supremum of the conformal modules of annuli on which there exists a special holomorphic $(0,4)$-bundle that represents ${\sf Pr}(\widehat{\mathfrak{F}})$.
Statement (1) of Proposition \ref{propEl.2} implies the inequality $\mathcal{M}(\widehat{\mathfrak{F}})\leq \mathcal{M}({{\sf Pr}(\widehat{\mathfrak{F}})})$, Statement (2) implies the opposite inequality. The corollary is proved. \hfill $\Box$

\section{The Gromov-Oka Principle for $(1,1)$-bundles over tori
with a hole}\label{sec:9.8}
In the following Theorem \ref{thmEl.9} we consider isotopies of smooth $(1,1)$-families
over tori with a hole to complex analytic $(1,1)$-families.

\begin{thm}\label{thmEl.9}
$(1)$ Let $X$ be a smooth surface of genus one with a hole with base point $x_0$, and with a chosen set
$\mathcal{E}=\{e_1,e_2\}$ of generators of $\pi_1(X,x_0)$. Define the set
$\mathcal{E}_0=
\{e_1,e_2, e_1 e_2^{-1},  e_1 e_2^{-2}, e_1e_2 e_1^{-1} e_2^{-1}\}$ as in Theorem {\rm \ref{thmGrom1}}. Consider a smooth $(1,1)$-bundle $\mathfrak{F}$
on $X$. Suppose for each $e\in \mathcal{E}_0$ the restriction of the bundle $\mathfrak{F}$ to an annulus representing $\hat e$ has the Gromov-Oka property.
Then the bundle $\mathfrak{F}$ has the Gromov-Oka property on $X$.\\
$(2)$ If a bundle $\mathfrak{F}$ as in $(1)$ is irreducible, then it is isotopic
to an isotrivial bundle, and hence, for each conformal structure $\omega$ on $X$ the bundle
$\mathfrak{F}_{\omega}$ is isotopic
to a bundle that extends to a holomorphic $(1,1)$-bundle on
the closed torus $\omega(X)^{c}$. In particular, $\mathfrak{F}$ is isotopic to a holomorphic bundle for any conformal structure $\omega$ on $X$ (maybe, of first kind). \\
$(3)$ Any smooth reducible bundle $\mathfrak{F}$ on $X$ has a single irreducible bundle component. This
irreducible bundle component is isotopic
to an isotrivial bundle. There is a Dehn twist in the fiber over the base point such that the
$(1,1)$-bundle $\mathfrak{F}$ can be
recovered from the irreducible bundle component up to composing each monodromy by a power of the Dehn twist. 
Any smooth reducible bundle on $X$ is isotopic to a holomorphic $(1,1)$-bundle for each conformal structure of second kind on $X$.\\
$(4)$ A reducible holomorphic $(1,1)$-bundle over a punctured Riemann surface is locally holomorphically trivial.
\end{thm}

\noindent {\bf Proof.}
By
Proposition \ref{propEl.2}, Statement (1), we may assume that $\mathfrak{F}$ is the canonical double branched covering of a smooth special $(0,4)$-bundle, denoted by ${\sf Pr}(\mathfrak{F})$.
Suppose that for each $e\in \mathcal{E}_0$
the restriction of the smooth $(1,1)$-bundle $\mathfrak{F}$ to an annulus $A_{\hat{e}}$ representing $\hat e$ has the Gromov-Oka property. Then for each $e\in \mathcal{E}_0$ there is a conformal structure $\omega_e:A_{\hat{e}}\to \omega_e(A_{\hat{e}})$ such that
$\omega_e(A_{\hat{e}})$ has conformal module bigger than $\frac{\pi}{2}\log(\frac{3+\sqrt{5}}{2})^{-1}$ and the pushed forward bundle
$(\mathfrak{F}|A_{\hat{e}})_{\omega_e}$
is isotopic
to a holomorphic bundle $\mathfrak{F}_e$  on $\omega_e( A_{\hat{e}}  )$.
By Proposition \ref{propEl.2a} we may assume that
$\mathfrak{F}_e$ is of the form \eqref{eqEl11e}, and by Proposition
\ref{propEl.2} we may assume that
$\mathfrak{F}_e$ is the canonical double branched covering of a holomorphic special $(0,4)$-bundle on $\omega_e(A_{\hat{e}})$,
denoted by ${\sf Pr}(\mathfrak{F}_e)$.
Since the monodromy homomorphisms of $\mathfrak{F}|A_{\hat{e}}$ and $\mathfrak{F}_e$ are conjugate,
the monodromy homomorphisms of ${\sf Pr}(\mathfrak{F}|A_{\hat{e}})$ and ${\sf Pr}(\mathfrak{F}_e)$ are also conjugate. Hence, ${\sf Pr}(\mathfrak{F}|A_{\hat{e}})$ and ${\sf Pr}(\mathfrak{F}_e)$ are isotopic.
We saw that for each $e\in \mathcal{E}$ the restriction of the special $(0,4)$-bundle ${\sf Pr}(\mathfrak{F})$ to an annulus representing $\hat e$ is isotopic to a holomorphic
bundle for a conformal structure on the annulus of conformal module bigger than $\frac{\pi}{2}\log(\frac{3+\sqrt{5}}{2})^{-1}$.
By Theorem \ref{thmEl.0} the special $(0,4)$-bundle ${\sf Pr}(\mathfrak{F})$ is isotopic to a holomorphic special $(0,4)$-bundle for any conformal structure of second kind on $X$.
Hence, the canonical double branched covering of ${\sf Pr}(\mathfrak{F})$ is isotopic to a holomorphic $(1,1)$-bundle for each conformal structure of second kind on $X$.
Therefore, $\mathfrak{F}$ is isotopic to a holomorphic $(1,1)$-bundle for each conformal structure of second kind on $X$.
Statement (1) is proved.

Suppose a bundle $\mathfrak{F}$ as in (1) is irreducible. Then by Proposition \ref{propEl.2}
the special $(0,4)$-bundle ${\sf{P}}(\mathfrak{F})$ is also irreducible. Moreover,
by Theorem \ref{thmEl.0}, Statement (1), the $(0,4)$-bundle ${\sf{P}}(\mathfrak{F})$
is isotopic to an isotrivial bundle, i.e. its lift to a finite covering $\hat X$ is isotopic to the trivial bundle, and therefore, for each conformal structure $\omega$ on $X$
(maybe, of first kind)
the $(0,4)$-bundle ${\sf{P}}(\mathfrak{F})_{\omega}$ is isotopic to a bundle
that extends holomorphically to the closed torus $\omega(X)^c$. Then also the canonical double branched covering of the $(0,4)$-bundle is isotopic to an isotrivial bundle.
Hence, the bundle $\mathfrak{F}$ is isotopic to an isotrivial bundle and extends to a holomorphic bundle on the closed torus. In particular, for each conformal structure (maybe, of first kind) on $X$ the bundle $\mathfrak{F}_{\omega}$ is isotopic to a holomorphic bundle on $\omega(X)$.
This gives statement (2).

Suppose now $\mathfrak{F}$ is any smooth reducible bundle on $X$. Then
by Proposition \ref{propEl.2} the bundle ${\sf Pr}(\mathfrak{F})$ is reducible. By Theorem  \ref{thmEl.0} the special $(0,4)$-bundle ${\sf Pr}(\mathfrak{F})$ is isotopic to a holomorphic special $(0,4)$-bundle for any conformal structure of second kind on $X$. Hence, for any conformal structure of second kind on $X$ the bundle $\mathfrak{F}$  is isotopic to a holomorphic $(1,1)$-bundle which is a double branched covering.

Each admissible set of curves on a punctured torus consists of exactly one curve. Hence, there is an admissible closed curve $\gamma$ in the punctured
fiber $\mathcal{P}^{-1}(x_0)\setminus \mathbold{s}^{\infty}$ over the base point $x_0$,
that is mapped by the monodromy mapping class along each curve $\gamma_j$ representing an element $e_j\in \mathcal{E}$ to a curve that is isotopic to $\gamma$ in the punctured fiber over $x_0$. Choose for each monodromy mapping class a representative that maps the curve $\gamma$ onto itself.

The complement of $\gamma$ on the punctured torus $\mathcal{P}^{-1}(x_0)\setminus \mathbold{s}^{\infty}$ is connected and homeomorphic to the thrice punctured Riemann sphere. This implies first that there is a single irreducible component of $\mathfrak{F}$. Further, the restrictions of suitable representatives of the monodromy mapping classes to the complement of $\gamma$ on the punctured torus are conjugate to self-homeomorphisms of the thrice punctured Riemann sphere. The conjugated homeomorphisms extend to self-homeomorphisms
of the Riemann sphere with three distinguished points. The self-homeomorphisms may interchange the two points that come from different edges of $\gamma$, but each of them fixes the third distinguished point. Except the identity there is only one isotopy class of such mappings and its square is the identity (in other words, the mapping class is the class of an involution).
Hence, the monodromy mapping classes of the irreducible bundle component are powers of a single periodic mapping class.
Therefore the irreducible bundle component is isotopic to an isotrivial bundle.

The monodromy mapping classes of the original bundle $\mathfrak{F}$ are powers of Dehn twists about $\gamma$,
maybe, composed with an involution, and up to powers of Dehn twists about $\gamma$ in the fiber $\mathcal{P}^{-1}(x_0)$ the bundle can be recovered from the irreducible bundle component.
Statement (3) is proved.

We prove now statement (4). Consider a reducible holomorphic (1,1)-bundle $\mathfrak{F}$
over a punctured Riemann surface of genus $1$. By Proposition \ref{propEl.2}
it is holomorphically isomorphic to the double branched covering of a holomorphic special $(0,4)$-bundle ${\sf Pr}(\mathfrak{F})$ over a punctured Riemann surface.
By Proposition \ref{propEl.2} the special holomorphic$(0,4)$-bundle ${\sf Pr}(\mathfrak{F})$ is also reducible.
By Theorem \ref{thmEl.0} the bundle ${\sf Pr}(\mathfrak{F})$ is holomorphically trivial.
Then the double branched covering $\mathfrak{F}$ of ${\sf Pr}(\mathfrak{F})$ is locally holomorphically trivial.
Theorem \ref{thmEl.9} is proved. \hfill $\Box$

\chapter{Fundamental groups and bounds for the extremal length}

\label{chapter3-braids}
\setcounter{equation}{0}

In Chapter \ref{chapter4} we computed several versions of the extremal length of certain elements of the fundamental group of the twice punctured complex plane. It becomes impracticable to compute the versions of extremal length for all elements regardless of their ''complexity''. In this chapter we give effective upper and
lower bounds for the extremal length with totally real horizontal boundary values of any element of the fundamental group of the twice punctured complex plane. The bounds differ by a multiplicative constant not depending on the element.
The estimates are provided in terms of a natural syllable decomposition
of the reduced word representing the element. Estimates for the extremal length of  conjugacy classes of elements of the fundamental group (i.e. of the invariants mentioned by Gromov) are also obtained.
The extremal length with totally real horizontal boundary values
is capable to give more subtle information regarding Gromov's Oka Principle and limitations of its validity than the respective invariant of conjugacy classes.

In the last section (Section \ref{sec:3-braids5})
we give estimates for the extremal length
with totally real horizontal boundary values of $3$-braids,
and estimates of the extremal length of conjugacy classes of $3$-braids (and, hence, of the entropy of $3$-braids).

\section{The fundamental group and extremal length. Two Theorems. }
\label{sec:3-braids1}

It will be convenient here to normalize the twice punctured complex plane as $\mathbb{C}\setminus\{-1,1\}$. With this normalization the fundamental group $\pi_1\stackrel{def}=\pi_1(\mathbb{C}\setminus\{-1,1\},0)$ is a free group in two generators $a_j$, $j=1,2,$ where $a_1$ is represented by simple closed curves that surround $-1$ counterclockwise, and $a_2$ is represented by simple closed curves that surround $1$ counterclockwise.
Recall that there is a canonical isomorphism $\pi_1( \mathbb{C}\setminus\{-1,1\},0)\to \pi_1^{tr}\stackrel{def}=\pi_1( \mathbb{C}\setminus\{-1,1\},(-1,1))$.
The elements of the relative fundamental group $\pi_1( \mathbb{C}\setminus\{-1,1\},(-1,1))$ are homotopy classes of arcs in $ \mathbb{C}\setminus\{-1,1\}$, each of whose endpoints is contained in $(-1,1)$.
The isomorphism  $\pi_1( \mathbb{C}\setminus\{-1,1\},0)\to \pi_1^{tr}$ assigns to the element $e\in\pi_1( \mathbb{C}\setminus\{-1,1\},0)$, that is represented by a curve $\gamma$ with base point $0$, the
element  $e_{tr}\in \pi_1^{tr}=\pi_1( \mathbb{C}\setminus\{-1,1\},(-1,1))$, that is represented by $\gamma$.
The extremal length $\Lambda(e_{tr})$ of the homotopy class $e_{tr}$ in the sense of Definition \ref{def4.1} is called the extremal lengths of $e$ with totally real horizontal boundary values and is sometimes denoted by $\Lambda_{tr}(e)$. Recall that $\Lambda(e_{tr})$ is the infimum of the extremal length of all rectangles that admit a holomorphic mapping into $\mathbb{C}\setminus\{-1,1\}$ whose restriction to each maximal vertical segment represents $e_{tr}$. The name ''totally real''
refers to the notion for $3$-braids which is related to the present notion. The conformal module $\mathcal{M}(e_{tr})=(\Lambda(e_{tr}))^{-1}$ is also denoted by $\mathcal{M}_{tr}(e)$.
\index{$\pi_1(\mathbb{C}\setminus\{-1,1\},0)$} \index{$\pi_1$}
\index{$\pi_1(\mathbb{C}\setminus\{-1,1\},(-1,1))$}
\index{$\pi_1^{tr}$} \index{$a_j$} \index{$e_{tr}$} \index{$\Lambda_{tr}(e)$} \index{$\Lambda_{tr}(e)$} \index{$\mathcal{M}_{tr}(e)$} \index{$\mathcal{M}_{pb}(e)$} \index{extremal length of elements of $\pi_1$ ! with $tr$ horizontal boundary values}
\index{extremal length of elements of $\pi_1$ ! with $pb$ horizontal boundary values} \index{extremal length of elements of $\pi_1$ ! with mixed horizontal boundary values}

We will also consider the relative fundamental group  $\pi_1^{pb}\stackrel{def}=\pi_1( \mathbb{C}\setminus\{-1,1\}, i\mathbb{R})$ whose elements are homotopy classes of arcs in $\mathbb{C}\setminus\{-1,1\}$ with endpoints on the imaginary axis $i\mathbb{R}$. The canonical group isomorphism
$\pi_1( \mathbb{C}\setminus\{-1,1\},0)\to \pi_1^{pb}\stackrel{def}=\pi_1( \mathbb{C}\setminus\{-1,1\}, i\mathbb{R})$ assigns to each element $e$ of $\pi_1( \mathbb{C}\setminus\{-1,1\},0)$ represented by a curve $\gamma$ with base point $0$ the
element $e_{pb}$ of $\pi_1( \mathbb{C}\setminus\{-1,1\}, i\mathbb{R})$ represented by $\gamma$. The extremal length $\Lambda(e_{pb})$ in the sense of Definition \ref{def4.1} is called the extremal length with $pb$ boundary values of $e$ and is also denoted by $\Lambda_{pb}(e)$. The conformal module $\mathcal{M}(e_{pb})=(\Lambda(e_{pb}))^{-1}$ is also denoted by $\mathcal{M}_{pb}(e)$.
("$pb$" stands for ''perpendicular bisector''. In fact, the imaginary axis is the
perpendicular bisector of the line segment $(-1,1)$.)

More generally, it will be sometimes convenient to use the fact that $\pi_1( \mathbb{C}\setminus\{-1,1\},0)$ is canonically isomorphic to the relative fundamental group $\pi_1^+\stackrel{def}=\pi_1( \mathbb{C}\setminus\{-1,1\}, i\mathbb{R}\cup (-1,1))$ whose elements are homotopy classes of arcs in $\mathbb{C}\setminus\{-1,1\}$ with endpoints  in $ i\mathbb{R}\cup (-1,1))$. The group isomorphism assigns to each element $e$ of $\pi_1( \mathbb{C}\setminus\{-1,1\},0)$, represented by a closed curve $\gamma$, the
element $e_{+}$ of $\pi_1( \mathbb{C}\setminus\{-1,1\}, i\mathbb{R}\cup (-1,1))$, represented by $\gamma$.
For each element $e\in \pi_1( \mathbb{C}\setminus\{-1,1\},0)$ all its representatives are contained in both classes, $e_{tr}$ and $e_{pb}$, associated to $e$, and all representatives of $e_{tr}$ and $e_{pb}$ are contained in the element $e_+$ of $\pi_1^+$ that is associated to $e$.

We will identify each element of  $\pi_1(\mathbb{C}\setminus\{-1,1\},0)$ with the reduced word in the generators representing it. A word $w$ in the generators is reduced if it is the identity or it has the form
$w= w_1^{n_1} \cdot
w_2^{n_2} \cdot \ldots ,\,$ where the $n_j$ are non-zero integers and
the $w_j$
are alternately equal to either $a_1$ or $a_2$. We refer to the
$w_j^{n_j}$ as the terms of the word.
\index{$e_{pb}$} \index{$\Lambda(e_{pb})$} \index{$\Lambda_{pb}(e)$} \index{bisector ! perpendicular} \index{$\pi_1^+$} \index{$\pi_1^{pb}$}

Any reduced word $\;w\;$ in  $\;\pi_1(\mathbb{C}\setminus \{-1,1\},0)\;$ can be uniquely decomposed into syllables. They are defined as follows.
\begin{defn}\label{defnxxx}
The syllables of any reduced word $w\in \pi_1(\mathbb{C}\setminus \{-1,1\},0)\;$ are all its terms $a_{j_i}^{k_i}$ with $|k_i|\geq 2$ (called syllables of form $(1)$), and all maximal sequences of consecutive terms  $a_{j_i}^{k_i}$  for which $|k_i|=1$ and all $k_i$ have the same sign. A syllable of the latter form is called a syllable of form $ (2)$, if it contains more than one term and is called a singleton or a syllable of form $(3)$ if it consists of a single term.
\end{defn}
\index{syllable} \index{syllable ! degree of} \index{word ! reduced} \index{word ! terms of} \index{word ! syllables of} \index{word ! elementary} \index{singleton}
(See also \cite{Jo2}, \cite{Jo3}). Define the degree of a syllable $\mbox{deg}(\mbox{syllable})$
to be the sum of the absolute values of the powers of terms entering
the syllable.
For example, the syllables of the word $a_2^{-1}\,a_1^2\,
a_2^{-3}\,a_1^{-1}\,a_2^{-1}\,a_1^{-1}\,a_2\,a_1^{-1}$  from
left
to right are the singleton $a_2^{-1}$, the syllable $a_1^2$ of form $(1)$ and
degree $2$, the
syllable $a_2^{-3}$ of form (1) and degree $3$, the syllable
$a_1^{-1}\,a_2^{-1}\,a_1^{-1}$ of form $(2)$ and
of degree $3$, the singleton $a_2$ and the singleton
$a_1^{-1}$.

We call words consisting of a single syllable elementary words.
Label the syllables of a non-elementary word from left to right by consecutive
integral numbers $j=1,2,\ldots \;$ . Let
$d_j$ be the degree of the $j$-th syllable $\mathfrak{s}_j$.
(We consider each syllable as a reduced word in the elements of the fundamental group.)
\index{$\mathcal{L}(w)$} \index{$\mathfrak{s}_j$}
Put
\begin{equation}\label{eq3+}
\mathcal{L}(w)\stackrel{def}=   \sum_j \log(2 d_j+\sqrt{4 d_j^2-1})\,,
\end{equation}
where the sum runs over the degrees of all syllables of $w$.
If $w$ is the identity we put $\mathcal{L}(w)=0$.
We want to point out that  $\mathcal{L}(w^{-1})=\mathcal{L}(w)$.
Notice that for the word consisting of the single syllable $\mathfrak{s}_j$ we have $\mathcal{L}(\mathfrak{s}_j)= \log ( 2 d_j+\sqrt{4 d_j^2-1}   )$. Thus, $\mathcal{L}(w) = \sum \mathcal{L}(\mathfrak{s}_j)$ where the sum runs over the syllables of $w$.

In this chapter we will prove the following theorems. The present proof of the upper bound is shorter and less technical than the proof in \cite{Jo2}, but the proof in  \cite{Jo2} gives a better constant.
\begin{thm}\label{thm1} For any element $w\in \pi_1(\mathbb{C}\setminus \{-1,1\},0)$ the following estimates hold.
\begin{equation}\label{eq1a}
\frac{1}{2\pi} \, \mathcal{L}(w) \leq \Lambda_{tr}(b)
=\frac{1}{\mathcal{M}_{tr}(w)} \leq  2^{14}\cdot
\mathcal{L}(w),
\end{equation}
except  in the following cases: $w=a_1^n$ or
$w=a_2^n$
for an integer $n$. In these cases $\Lambda_{tr}(b)=0$, i.e
$\mathcal{M}_{tr}(b)=\infty$.

Moreover,
\begin{equation}\label{eq1b}
\frac{1}{2\pi} \, \mathcal{L}(w) \leq \Lambda_{pb}(w)=
\frac{1}{\mathcal{M}_{pb}(w)} \leq   2^{13}
\cdot \mathcal{L}(w),
\end{equation}
except in the following case:
each term in the reduced word $w$ has the same power, which
equals either $+1$ or $-1$. In these cases $\Lambda_{pb}(w)=0$,
i.e. $\mathcal{M}_{pb}(w)=\infty$.
\end{thm}

We will also consider the relative fundamental
groups $^{pb}\pi_1^{tr}$ and $^{tr}\pi_1^{pb}$ with mixed horizontal
boundary values.
Here the elements of $^{pb}\pi_1^{tr}$ are the homotopy
classes $_{i\mathbb{R}}\textsf{h}_{(-1,1)}$ and the elements of
$^{tr}\pi_1^{pb}$ are the homotopy classes
$_{(-1,1)}\textsf{h}_{i\mathbb{R}}$ in the space $X=\mathbb{C} \setminus \{-1,1\}$ .
There is a canonical isomorphism from $\pi_1$ onto $^{pb}\pi_1^{tr}$ (onto $^{tr}\pi_1^{pb}$, respectively) that assigns to an element $e\in \pi_1$  the element $_{pb}e_{tr}$ that contains all representatives of $e$ (the element $_{tr}e_{pb}$, respectively, that contains all representatives of $e$).
\index{$^{pb}\pi_1^{tr}$}   \index{$^{tr}\pi_1^{pb}$} \index{$_{pb}e_{tr}$}
\index{$_{tr}e_{pb}$}

The statement for mixed boundary values is given in the following Theorem \ref{thm10.1'}. Note that for mixed boundary values the estimate of the extremal length holds always while for $pb$ or $tr$ boundary values there are exceptional cases.

\medskip
\begin{thm}\label{thm10.1'}
For all $w \in \pi_1$ the following inequalities hold
\begin{align}
\frac{1}{2\pi} \, \mathcal{L}(w) \leq \Lambda(_{tr}w_{pb})
\leq   2^{14} \cdot
\mathcal{L}(w), \label{eq13a}\\
\frac{1}{2\pi} \, \mathcal{L}(w) \leq \Lambda(_{pb}w_{tr})
\leq   2^{14}
\cdot \mathcal{L}(w).\label{eq13b}
\end{align}
\end{thm}
\medskip

Recall that the conjugacy class $\hat w$ of an element $w\in \pi_1(\mathbb{C}\setminus\{-1,1\},0)$ can be identified with the free homotopy class of closed curves that contains $w$.

\begin{defn}\label{defnxxx'}
A word $w \in\;\pi_1(\mathbb{C}\setminus \{-1,1\},0)\;$  is called cyclically reduced, if either the word consists of a single term, or it has at least two terms and the first and the last term of the word are powers of different generators. Each reduced word which is not the identity is conjugate to a cyclically reduced word.
\end{defn}
\index{word ! cyclically reduced}

The following theorem gives upper and lower bounds of the invariant $\mathcal{M}(\hat w)$  of conjugacy classes of elements of the twice punctured complex plane, that was first mentioned in Gromov's seminal paper \cite{G} in connection with obstructions for the Gromov-Oka principle.
Each element of the fundamental group $\pi_1(\mathbb{C}\setminus\{-1,1\},0)$ corresponds to an element of $\mathcal{PB}_3\diagup\mathcal{Z}_3$ and, hence, to a mapping class of the twice punctured complex plane.
By the Main Theorem the conformal module
$\mathcal{M}(\hat w)$ can be interpreted as a multiple of the inverse of the entropy of the conjugay class
$\hat w$.

\begin{thm}\label{thm2} Let $\hat w$ be a conjugacy class of elements of $\pi_1(\mathbb{C}\setminus\{-1,1\},0)$,
and let $w$ be a
cyclically reduced word representing the conjugacy class $\hat w$. Then
$$
\frac{1}{2 \pi} \cdot \mathcal{L}(w) \,\leq \, \Lambda (\hat w)
\, \leq \,   2^{13} \cdot \mathcal{L}(w),
$$
with the following exceptions: $\hat{w}=\widehat{a_1^n}$, $\hat{w}=\widehat{a_2^n}$ and
$\hat{w}=\widehat{(a_1a_2)^n}$. In the exceptional cases $\Lambda(\hat w)= 0$
and
$\mathcal{M}(\hat w) =\infty$.
\end{thm}

\section{Coverings of $\mathbb{C} \setminus \{-1,1\}$ and
slalom curves}\label{sec:3-braids1a}
For the proof of the theorems 
it will be convenient to use instead
of the universal covering of $\mathbb{C} \setminus \{-1,1\}$
two different coverings of $\mathbb{C} \setminus \{-1,1\}$ by
$\mathbb{C}
\setminus i \mathbb{Z}$. Though the universal covering is well
studied,
factorizing the universal covering through two different coverings by
$\mathbb{C} \setminus i \mathbb{Z}$ has the advantage to make the
contribution
of the syllables to the extremal length transparent.

To obtain the first covering $\mathbb{C} \setminus i \mathbb{Z} \to
\mathbb{C}
\setminus \{-1,1\}$ we take the
universal covering of the twice punctured Riemann sphere
$\mathbb{P}^1 \setminus \{-1,1\}$ (the logarithmic covering) and
remove all preimages of $\infty$ under the covering map.
Geometrically the logarithmic covering of $\mathbb{P}^1
\setminus
\{-1,1\}$ can be described as follows. Take copies of
$\mathbb{P}^1
\setminus[-1,1]$ labeled by the set $\mathbb{Z}$ of integer
numbers. Attach to each copy two copies of
$(-1,1)$, the
$+$-edge (the set of accumulation points contained in $(-1,1)$ of the upper
half-plane)
and the $-$-edge (the set of  accumulation  points  contained in $(-1,1)$ of the lower
half-plane). For each $k \in \mathbb{Z}$ we glue the $+$-edge
of the
$k$-th copy to the $-$-edge of the $k+1$-st copy (using the
identity
mapping on $(-1,1)$ to identify points on different edges).
Denote
by $U_{\log}$ the set obtained from the described covering by
removing all preimages of $\infty$. \index{$U_{\log}$}

Choose curves $\alpha_j(t), t \in [0,1],\, j=1,2,$  in
$\mathbb{C}
\setminus \{-1,1\}$  that represent the
generators $a_j,j=1,2,$ of the fundamental group $\pi_1(\mathbb{C}
\setminus \{-1,1\},\,0\,\})$ and have the following properties. The
initial point and the terminal point of each of the curves are the
only points of the curve on the
interval $(-1,1)$ and are also the only points of the curve on the
imaginary axis.
The following proposition holds.

\begin{prop}\label{prop2} The set $U_{\log}$ is conformally
equivalent to $\mathbb{C}\setminus  i\,\mathbb{Z}$. The
mapping
$f_1 \circ f_2$, $f_2(z)= \frac{e^{\pi z} -1}{e^{\pi z} +1},
\, z
\in
\mathbb{C}\setminus  i\,\mathbb{Z}$, $f_1(w)=
\frac{1}{2}(w+\frac{1}{w}),\, w \in \mathbb{C}\setminus\{-1,0,1\}$,
is
a covering map from $\mathbb{C}\setminus  i \,\mathbb{Z}$ to
$\mathbb{C} \setminus \{-1,1\}$.  The mapping $f_1 \circ f_2$  has period $i$ and takes the punctured strip $\{z\in \mathbb{C}: -\frac{1}{2}  <  {\rm Im} z < \frac{1}{2}\}\setminus \{0\}$ conformally onto
$\mathbb{C}\setminus [-1,1]$ mappings points in the upper half-plane to the lower half-plane and vice versa. It maps the set $(-\frac{i}{2},\frac{i}{2})\setminus \{0\}$ onto $i \mathbb{R}\setminus\{0\}$,
and the line $\mathbb{R}+\frac{i}{2}$ onto $(-1,1)$.

For each $k \in \mathbb{Z}$ the lift of $\alpha_1$ with initial point $\frac{- i}{2} +  i
k$ is a curve which joins $\frac{- i}{2} +  i k$ with
$\frac{- i}{2} +  i (k+1)$ and is contained in the closed
left half-plane. The only points on the imaginary axis are
the
endpoints.

The lift of $\alpha_2$ with initial point $\frac{- i}{2} +  i
k$ is a curve which joins $\frac{- i}{2} +  i k$ with
$\frac{- i}{2} +  i (k-1)$ and is contained in the closed
right half-plane. The only points on the imaginary axis are
the
endpoints.
\end{prop}

Figure 1 below shows the curves $\alpha_1$ and $\alpha_2$ which represent the
generators of the fundamental group $\pi_1(\mathbb{C} \setminus
\{-1,1\} ,0)$
and their lifts under the covering maps $f_1$ and $f_2 \circ f_1$ . For
$j=1,2$
the curves $\alpha_j '$ and $\alpha_j ''$ are the two lifts of
$\alpha_j$
under the double covering $f_1: \mathbb{C} \setminus \{-1,0,1\} \to
\mathbb{C}  \setminus \{-1,1\}$. The curve $\hat {\alpha}
_1'$ is the lift of $\alpha _1'$ under the mapping $f_2$ with initial
point
$\frac{-i}{2}$, the curve $\hat {\alpha} _1''$ is
the lift of $\alpha _1''$ under the mapping $f_2$ with initial point
$\frac{i}{2}$,
the curve $\hat {\alpha} _2'$ lifts $\alpha _2'$ and has
initial point $\frac{i}{2}$, and the curve $\hat {\alpha} _2''$ lifts
$\alpha
_2''$ and has
initial point $-\frac{i}{2}$.

\medskip

\noindent \textbf{Proof}.
The mapping $f_1$ is the restriction of the Zhukovsky function to
$\mathbb{C}\setminus\{-1,0,1\}$. The
Zhukovski function defines a double
branched covering of the Riemann sphere $\mathbb{P}^1$ with branch
locus $\{-1,1\}$. In particular, the Zhukovsky function provides a
conformal mapping
from the unit disc
$\mathbb{D}$ onto $\mathbb{P}^1 \setminus [-1,1]$. It maps
$-1$ to
$-1$, $1$ to $1$ and $0$ to $\infty$. The upper half-circle is
mapped onto the $-$-edge, the upper half-disc is mapped onto the
lower
half-plane, the lower half-circle is mapped onto the $+$-edge
and the
lower half-disc is mapped onto the upper half-plane.
Similarly, it provides a conformal mapping of the exterior of the
closed unit
disc onto $\mathbb{P}^1 \setminus [-1,1]$ which preserves the upper
half-plane
and also preserves the lower half-plane. \index{function ! Zhukovski}

The mapping $f_2$ extends through $i\mathbb{Z}$ to an infinite covering of $\mathbb{P}^1
\setminus \{-1,1\}$ by $\mathbb{C}$. By an abuse of notation we denote this extension also by $f_2$. (The extension of) $f_2$ provides a
conformal mapping $f_2^0$ from the strip $\{z \in \mathbb{C}:
-\frac{1}{2}
< \mbox{Im} \, z < \frac{1}{2}\}$ onto the unit disc, which takes the real
axis onto the segment $(-1,1)$ , such that
$\lim _{x \in \mathbb{R}, x \to -\infty}= -1$, $\lim _{x \in
\mathbb{R}, x
\to +\infty}=1$. Further, $f_2$ maps
$\frac{i}{2}$ to $i$, $-\frac{i}{2}$ to $-i$, and $0$ to $0$.
The line $\{z \in \mathbb{C}: \mbox{Im} \, z
=\frac{1}{2}\}$ is mapped onto the upper half-circle, the upper
half-strip $\{z \in \mathbb{C}: 0 < \mbox{Im}\, z <
\frac{1}{2}\}$ is mapped onto the upper half-disc, the line $\{z
\in
\mathbb{C}: \mbox{Im} \, z =-\frac{1}{2}\}$ is mapped onto the lower
half-circle and the lower half-strip is mapped onto the lower
half-disc. It follows that $f_1\circ f_2$ takes the punctured strip  $\{z \in (\mathbb{C}\setminus \{0\}):-\frac{1}{2}< \mbox{Im} \, z < \frac{1}{2}\}$ conformally onto $\mathbb{C}\setminus [-1,1]$.

\begin{figure}[h]
\begin{center}
\includegraphics[width=11cm]{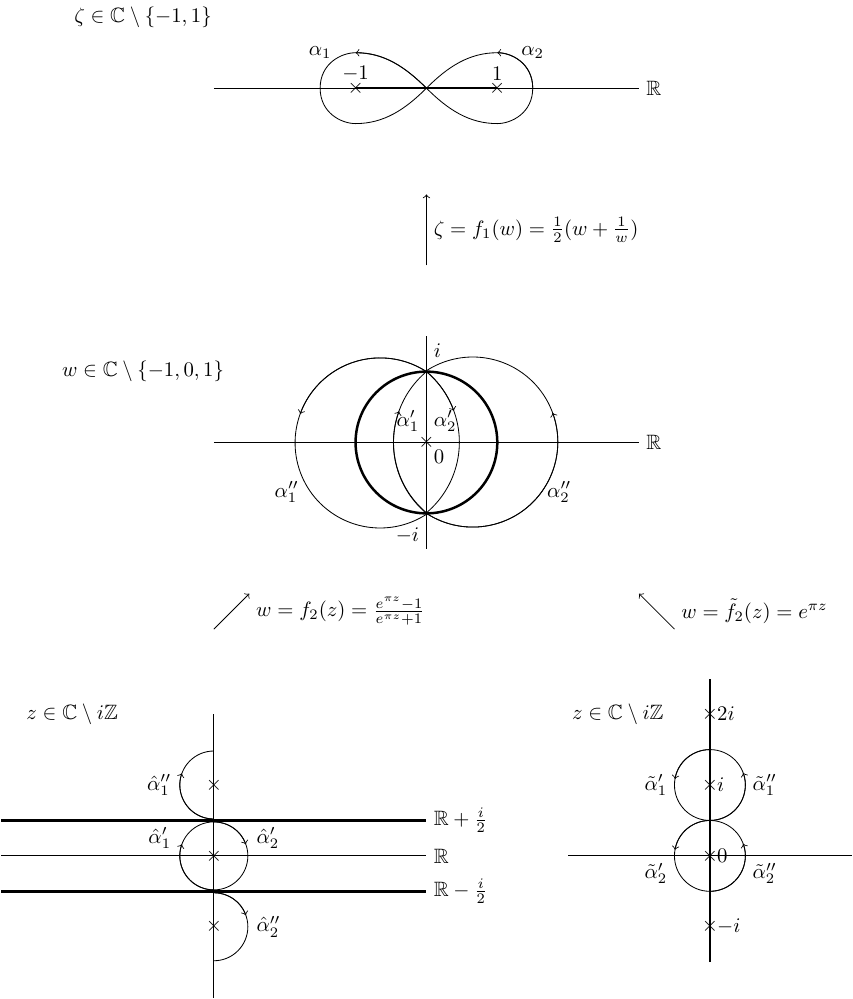}
\end{center}
\caption{The generators of $\pi_1(\mathbb{C}\setminus\{-1,1\},0)$ and their lifts under two
coverings by $\mathbb{C}\setminus i\mathbb{Z}$ .}
\label{Fig3-braids1}
\end{figure}

The mapping $f=f_1 \circ f_2$ has period $i$. Indeed, $f_2$ has period $2i$,
$f_2(z+i)=\frac{1}{f_2(z)}$, and $f_1(\frac{1}{w})=f_1(w)$. The statements concerning the lift of curves
under $f_1 \circ f_2$ are now clear.

To see that $U_{\log}$ is conformally equivalent to
$\mathbb{C}
\setminus  i\,\mathbb{Z} $ we identify the set $\mathbb{C} \setminus
[-1,1] $ with the sheet of $U_{\log}$ labeled by $0$. The
conformal
mapping $f \mid \{z \in \mathbb{C} \setminus \{0\}:- \frac{1}{2} <
\mbox{Im} \, z
< \frac{1}{2}\}$ (whose image is $\mathbb{C} \setminus
[-1,1] $ ) extends by Schwarz's Reflection Principle through
the line $\{z \in \mathbb{C}: \mbox{Im} \, z =\frac{1}{2}\}$ which is
mapped onto the $-$-edge of $\mathbb{C} \setminus [-1,1] $. The
reflected mapping takes
$\{z \in \mathbb{C}\setminus\{i\}: \frac{1}{2} < \mbox{Im} \, z
< \frac{3}{2}\}$ conformally onto  $\mathbb{C} \setminus
[-1,1] $.
We identify the image of the punctured strip $\{z \in \mathbb{C}\setminus \{i\}:
\frac{1}{2} < \mbox{Im} \, z
< \frac{3}{2}\}$ with the sheet
of $U_{\log}$ labeled by $-1$.
Induction
on reflection through the lines  $\{z \in \mathbb{C}:
\mbox{Im} \,z
=\frac{1}{2} +j\},\,j \in \mathbb{Z},$ gives the conformal
mapping
from $C \setminus  i \mathbb{Z} $ onto $U_{\log}$ .

The remaining statements are easy to see.
\hfill $\Box$

\medskip

The second covering is given by the mapping $f_1\circ \tilde f_2:
\mathbb{C}
\setminus i \mathbb{Z} \to \mathbb{C} \setminus\{-1,1\}$ , where $f_1$
is as
before and $\tilde f_2$ is the exponential map, $\tilde f_2(z)=e^{\pi
z}$. Recall that each curve $\alpha_j,\, j=1,2,$ has two lifts
$\alpha_j'$
and
$\alpha_j''$ under $f_1$. For an illustration of the following
proposition see the right part of Figure \ref{Fig3-braids1}.

\begin{prop}\label{prop5a} The mapping
$\tilde f_2$ takes $\mathbb{C} \setminus i\mathbb{Z}$ to
$\mathbb{C}\setminus
\{-1,0,1\}$, and $f_1$ takes the latter set to $\mathbb{C}\setminus
\{-1,1\}$. The composition $\tilde f_2\circ f_1$ is a covering of $\mathbb{C}\setminus \{-1,1\}$ by $\mathbb{C}\setminus
\{-1,0,1\}$ and has period $2i$. For each integer $k$ it takes the interval $(ki,(k+1)i)$ to $(-1,1)$, it maps $\frac{i}{2}+\mathbb{R}$ to $i(0,\infty)$, and  maps $-\frac{i}{2}+\mathbb{R}$ to $i(-\infty,0)$.

The lifts  $\tilde
\alpha_j'$ and
$\tilde \alpha_j''$, $j=1,2,$ of $\alpha_j'$ and
$\alpha_j''$ under $\tilde f_2$ have the following properties. The
lifts
$\tilde
\alpha_j',\, j=1,2,\,$   are contained in the closed left half-plane
and are
directed
downwards (i.e in the direction of decreasing $y$), the lifts  $\tilde
\alpha_j'',\,j=1,2,\,$ are contained in the closed right half-plane and
directed
upwards.
The initial point of $\tilde \alpha_1'$ is  $i + \frac{1}{2} i$, the
initial
point of $\tilde \alpha_1''$ and $\tilde \alpha_2'$ is
 $\frac{1}{2} i$, the initial point of $\tilde \alpha _2''$ is
 $-\frac{1}{2}
 i$. All other lifts are obtained by translation by an integer
 multiple of
 $2i$.
\end{prop}

The straightforward proof is left to the reader.

Consider the curve $\alpha_1^n,\;n\in \mathbb{Z}\setminus
\{0\}$. It
runs $n$ times along the curve $\alpha_1$ if $n>0$, and $|n|$
times
along the curve $\alpha_1$ with inverted orientation if $n<0$. It is
homotopic in $\mathbb{C} \setminus \{-1,1\}$ with base point
$0$ to a curve whose interior (i.e. the complement of its endpoints) is
contained in the open left
half-plane. We call a representative of $a_1^n$ whose interior is in the open left half-plane a standard representative of
$a_1^n$. In the same way we define standard
representatives of $a_2^n$.
For each $k \in
\mathbb{Z}$ the curve $\alpha_1^n$ lifts under $f_1\circ f_2$ to a
curve with
initial point
 $\frac{- i}{2} +  i k$ and terminal point  $\frac{- i}{2}
 +  i k +
 i n$
which is contained in the closed left half-plane and omits the
points in $ i \mathbb{Z}$. Respectively, $\alpha_2^n,\;n\in
\mathbb{Z}\setminus \{0\},$ lifts under $f_1\circ f_2$ to a curve with
initial
point
 $\frac{+ i}{2} +  i k$ and terminal point  $\frac{+ i}{2}
 +  i k -
 i n$
which is contained in the closed right half-plane and omits the
points in $ i \mathbb{Z}$. The mentioned lifts are homotopic
through curves in $\mathbb{C} \setminus  i \mathbb{Z}$ with
endpoints on $i\,\mathbb{R} \setminus  i \mathbb{Z}$ to curves
with interior contained either in the open right half-plane or in the open
left half-plane. Standard representatives of $a_1^n$ and $a_2^n$ lift to
such curves.
\index{$a_j$ ! standard representatives of}

\begin{defn}\label{def4} A simple curve in  $\mathbb{C}
\setminus  i
\mathbb{Z}$ with endpoints on different connected components
of
$i\mathbb{R}
\setminus  i \mathbb{Z}$ is called an elementary slalom curve
if
its interior is
contained in
one of the open half-planes $\{z \in \mathbb{C}: \,\mbox{Re} \, z
>0\}$ or
$\{z \in \mathbb{C}: \,\mbox{Re} \, z <0\}$.

A curve in  $\mathbb{C} \setminus  i \mathbb{Z}$ is called an
elementary half slalom curve if one of the endpoints is
contained in a
horizontal line $\{z \in \mathbb{C}:\mbox{Im} \, z =k+\frac{1}{2}\}$ for
an
integer $k$
and the union of the curve with its (suitably oriented)
reflection in the line
$\{z \in \mathbb{C}:\mbox{Im} \, z = k+\frac{1}{2}\}$
is an elementary slalom curve.

 A slalom curve in  $\mathbb{C}
\setminus  i \mathbb{Z}$ is a curve which can be divided into
a
finite number of elementary slalom curves so that consecutive
elementary slalom curves are contained in different
half-planes.

A curve which is homotopic to a slalom curve (elementary slalom curve,
respectively) in $\mathbb{C} \setminus  i
\mathbb{Z}$ through curves with endpoints in $i\mathbb{R}
\setminus
 i \mathbb{Z}$ is called a homotopy slalom curve (elementary homotopy
 slalom
 curve, respectively).

A curve which is homotopic  to an elementary half-slalom curve in
$\mathbb{C}
\setminus  i
\mathbb{Z}$ through curves with one endpoint in $i\mathbb{R}
\setminus
 i \mathbb{Z}$ and the other endpoint on the line
 $\{z \in \mathbb{C}:\mbox{Im} \,z =k+\frac{1}{2}\}$ for an
integer $k$ is called an elementary homotopy half-slalom curve.
\end{defn}
\index{curve ! slalom curve} \index{curve ! half-slalom curve}
\index{curve ! elementary slalom curve}
\index{curve ! elementary half-slalom curve}
\index{curve ! homotopy slalom curve}
\index{curve ! trivial homotopy slalom curve}

We call an elementary slalom curve non-trivial if its endpoints are
contained in
intervals $(ik,i(k+1))$ and $(i\ell, i(\ell+1))$ with $|k-\ell|\geq 2$.
Note
that the union of an elementary half-slalom curve with its
reflection in
the horizontal line that contains one endpoint is always a non-trivial
elementary slalom curve (i.e. an elementary half-slalom curve is "half" of a
non-trivial elementary slalom curve).
We saw that the lifts under $f_1 \circ f_2$  of representatives of
terms
$_{pb}(a_j^n)_{pb} \in \pi_1^{pb}$ with $|n|\geq 2$ are non-trivial
elementary
homotopy slalom curves.

The lifts of representatives of elements of $\pi_1$ under
$f_1\circ \tilde
f_2$ look different.
The representatives of $(a_j^n)_{tr}$, $|n|\geq 1,$ lift to curves which
make $|n|$ half-turns around a point in $i\mathbb{Z}$ (positive
half-turns if
$n>0$, and negative half-turns if $n<0$). Each such
representative lifts
under $f_1\circ \tilde f_2$ to the composition of $|n|$ trivial
elementary homotopy
slalom curves.

Take any syllable of from (2), i.e. any maximal sequence of at least two
consecutive terms of
the word which enter with equal power being either $1$ or $-1$. Recall
that $d$ denotes the sum of the absolute values of the powers of the terms of
the syllable. There is a representing curve that makes $d$ half-turns
around the interval $[-1,1]$ (positive half-turns, if the exponents of terms
in the
syllable are $1$, and negative half-turns otherwise). The lift under $f_1 \circ \tilde f_2$
of this representative is a non-trivial elementary slalom curve.

We will call homotopy classes of homotopy slalom curves for short slalom classes.

\begin{figure}[h]
\begin{center}
\includegraphics[width=8cm]{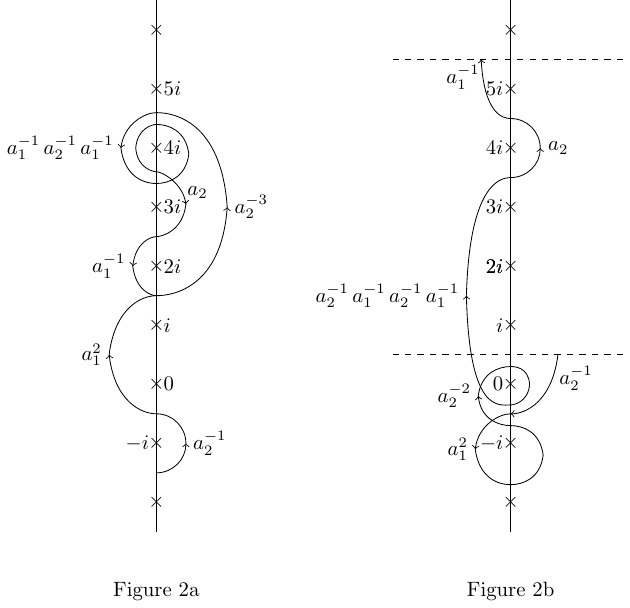}
\end{center}
\caption{Lifts of a curve in $\mathbb{C}\setminus  \{-1,1\}$ under two different coverings by $\mathbb{C}\setminus i\mathbb{Z}$. }\label{Fig3-braids2.pdf}
\end{figure}

Figure \ref{Fig3-braids2.pdf}
shows two slalom curves.  The curve in Figure 10.2a
is a lift under $f_1 \circ f_2$ of a  curve in
$\mathbb{C}\setminus  \{-1,1\}$ with initial  point and terminal
point equal to $0 \in i\mathbb{R}$ representing the word
$
a_2^{-1}\,a_1^2\,
a_2^{-3}\,a_1^{-1}\,a_2^{-1}\,a_1^{-1}\,a_2\,a_1^{-1}\,
$
in the relative fundamental group $\pi_1(\mathbb{C}\setminus
\{-1,1\},i\mathbb{R})$.
The curve in Figure 10.2b is a lift under $f_1 \circ \tilde f_2$ of a
curve with initial and terminal point in $i\mathbb{R} \setminus \{0\}$
representing the same word. For each elementary slalom curve or half-slalom curve in Figure  \ref{Fig3-braids2.pdf}   we indicate
the element of $\pi_1$ which lifts under the considered mapping to the respective elementary slalom class or half-slalom class.

Non-trivial elementary slalom classes and elementary half-slalom classes have positive
extremal length (in the sense of Definition \ref{def4.1})
which can be effectively estimated from above and from below.
In this sense homotopy classes of curves in $\mathbb{C}\setminus \{-1,1\}$ whose lifts
under $f_1 \circ f_2$ or $f_1 \circ \tilde f_2$  are
non-trivial elementary slalom classes or elementary half-slalom classes serve as building blocks. We obtained the following fact. For an element of a relative fundamental group of the
twice
punctured complex plane the pieces representing syllables of form (1)
with $pb$
boundary values lift to non-trivial elementary slalom classes
under
$f_1\circ f_2$, while the pieces representing syllables of form (2)
with $tr$
boundary values  lift to non-trivial elementary slalom classes
under
$f_1\circ \tilde f_2$.
For a word in $\pi_1$ containing a syllable that is a singleton we may select a piece of a representing curve that lifts under $f_1\circ f_2$ to a non-trivial elementary homotopy half-slalom curve. These facts will be
used to obtain a lower bound for the extremal length of elements of the
relative fundamental groups.
Recall that the extremal length of a homotopy class of curves is equal
to the
extremal length of the class of their lifts under a holomorphic
covering.

The method to obtain the upper bound is roughly to patch together in a
quasiconformal way the holomorphic mappings of rectangles representing
syllables and to perturb the obtained quasiconformal mapping to a holomorphic mapping.

We conclude this section with  relating the two explicitly given
coverings of
$\mathbb{C}\setminus \{-1,1\}$ by $\mathbb{C} \setminus i\mathbb{Z}$ to
the
universal covering of the twice punctured plane.
 Denote by ${\sf P}:U\to \mathbb{C}
\setminus  i\mathbb{Z}$ the
universal covering of $\mathbb{C}
\setminus  i\mathbb{Z}$. Geometrically the set $U$ is obtained in the
following
way. Consider the left half-plane $\mathbb{C}_{\ell}$ and call it the Riemann surface of generation $0$. The first step is the following. For each
integer $k$  we take a copy of the right
half-plane $\mathbb{C}_r$  and glue it to $\mathbb{C}_{\ell}$ along the
interval $(ki,(k +1)i)$ (using the identity mapping for gluing). We obtain a Riemann
surface with a natural projection to
$\mathbb{C} \setminus i \mathbb{Z}$ called the Riemann surface of first generation. At the second step we consider the
Riemann surface of first generation and proceed similarly by gluing
the left half-plane along each copy
of
intervals
$(ki, (k + 1)i)$ which is an end of the Riemann
surface of first generation. By
induction
we obtain the universal covering of $\mathbb{C} \setminus i
\mathbb{Z}$.
\index{$U$}

Denote by $\widetilde{\mathbb{C}}_{\ell}$ the lift under $\sf{P}$ of the left
half-plane to the
first sheet of $U$ over $\mathbb{C}_{\ell}$.
Let $\mathbb{C}_{\ell}^{Cl}$ be the closure of
$\mathbb{C}_{\ell}$
in $\mathbb{C} \setminus i \mathbb{Z}$,  let
$\widetilde{\mathbb{C}}_{\ell}^{Cl}$ be	the closure of
$\widetilde {\mathbb{C}}_{\ell}$  in the
universal
covering of $\mathbb{C} \setminus i \mathbb{Z}$ and let
${(ki, (k + 1)i)}{\;\widetilde{}} \subset \widetilde{\mathbb{C}}_{\ell}^{Cl}$
be
the
lift of the intervals $(ki, (k + 1)i)$.
\index{$\widetilde{\mathbb{C}}_{\ell}$} \index{$\widetilde{\mathbb{C}}_{\ell}^{Cl}$}

For each $k$ we denote by $\mathfrak{D}^{\ell}_k$  the half-disc
$\{z \in \mathbb{C}_{\ell}: |z-i(k+\frac{1}{2})|<\frac{1}{2}\}$ with
diameter
$(ik,i(k+1))$ which is contained in the left half-plane and by
$\rho_k$ the
respective open half-circle $\{z \in \mathbb{C}_{ \ell}:
|z-i(k+\frac{1}{2})|=\frac{1}{2}\}$. We call the $\rho_k$ half-circles of generation $0$.
\index{$\mathfrak{D}^{\ell}_k$} \index{$\rho_k$}

\begin{lemm}\label{lemm3} There is a conformal mapping
$\varphi : U
\to \mathbb{C}_{\ell}$ that takes
$\widetilde{\mathbb{C}}_{\ell}$ onto $\mathbb{C}_{\ell} \setminus
\bigcup
_{k=-\infty}^{\infty} \overline{\mathfrak{D}^{\ell}_k}$ so that for each $k$ the set
${(ki, (k + 1)i)}{\;\widetilde{}}$ is mapped onto $\rho_k$ under the continuous extension of $\varphi$.
\end{lemm}
\noindent {\bf Proof.} Consider the half-strip \index{$\mathfrak{H}_0$}
$\mathfrak{H}_0\stackrel{def}{=}\{z \in {\mathbb{C}}_{\ell}:
0<\mbox{Im} \, z<1\}$.
By a theorem of Caratheodory (\cite{Go}, Chapter II.3, Theorem 4 and Theorem $4'$,  and also \cite{Ma} , Theorem 2.24 and Theorem 2.25) each conformal mapping that  takes the lift
$\widetilde{\mathfrak{H}}_0$  of $\mathfrak{H}_0$ onto the set $\mathfrak{H}_0\setminus
\overline{\mathfrak{D}^{\ell}_0}$, extends continuously to a
homeomorphism
between closures. Let
$\varphi_0$  be the conformal mapping from $\widetilde{\mathfrak{H}}_0$   onto the set $\mathfrak{H}_0\setminus
\overline{\mathfrak{D}^{\ell}_0}$ whose extension to the boundary
takes the
point $\tilde 0$ over $0$ to $0$, the point $\tilde 1$ to $1$ and the
point
$\tilde{\infty}$ (considered as prime end of
$\widetilde{\mathfrak{H}}_0$)   to
$\infty$. The extension of the conformal mapping to the boundary takes
the lift
${\{z \in {\mathbb{C}}_{\ell}: \mbox{Im} \, z=1\}}{\;\widetilde{}}$
of ${\{z
\in
{\mathbb{C}}_{\ell}: \mbox{Im} \, z=1\}}$
onto ${\{z
\in
{\mathbb{C}}_{\ell}: \mbox{Im} \, z=1\}}$, it takes the lift ${\{z \in
{\mathbb{C}}_{\ell}: \mbox{Im} \, z=0\}}{\;\widetilde{}}$ of  ${\{z \in
{\mathbb{C}}_{\ell}: \mbox{Im} \, z=0\}}$ onto ${\{z \in {\mathbb{C}}_{\ell}:
\mbox{Im} \, z=0\}}$, and maps the lift ${(0, i)}{\;\widetilde{}}$ of $(0,i)$ onto $\rho_0$. By
induction we
extend $\varphi_0$ by Schwarz's Reflection Principle across the
half-lines
${\{z \in {\mathbb{C}}_{\ell}: \mbox{Im} \, z=k\}}{\;\widetilde{}}$ which are the lifts of the respective half-lines in ${\mathbb{C}}_{\ell}$. We obtain a conformal
mapping of $\widetilde{\mathbb{C}}_{\ell}$ onto $\mathbb{C}_{\ell}
\setminus
\bigcup _{k=-\infty}^{\infty} \overline{\mathfrak{D}^{\ell}_k}$, denoted again by $\varphi_0$, whose extension
to $\widetilde{\mathbb{C}}_{\ell}^{Cl}$  takes the segment
${(ik,
i(k+1))}{\;\widetilde{}}$ onto the half-circle $\rho_k$ for each integer number $k$. (See Figure \ref{Fig3-braids3.pdf}.)
\index{$\varphi_0$} \index{$\varphi$}

Schwarz's Reflection Principle across each segment  ${(ik,
i(k+1))}{\;\widetilde{}}$ provides an extension of the conformal mapping  $\varphi_0$ to the Riemann surface of first generation. Note that for each $k_0$ the image of the half-circles $\rho_{k}$, $k\neq k_0,$ under reflection of
$\mathbb{C}_{\ell} \setminus \bigcup _{k=-\infty}^{\infty}
\overline{\mathfrak{D}^{\ell}_k}$ across $\rho_{k_0}$
are half-circles with diameter on the imaginary axis. We call them half-circles of the first generation. We obtained a conformal mapping of the Riemann surface of first generation to the unbounded connected component of the left half-plane with the half-circles of first generation removed.
\index{Schwarz ! reflection principle}

Apply the reflection principle by induction to the Riemann surface of generation $n$ and all copies of intervals $(ik,i(k+1))$ in its "boundary". We obtain the Riemann surface of generation $n+1$, half-circles of generation  $n+1$ and a conformal mapping of this Riemann surface onto the unbounded connected component of the left half-plane with the half-circles of generation $n+1$ removed.

\begin{figure}[h]
\begin{center}
\includegraphics[width=4cm]{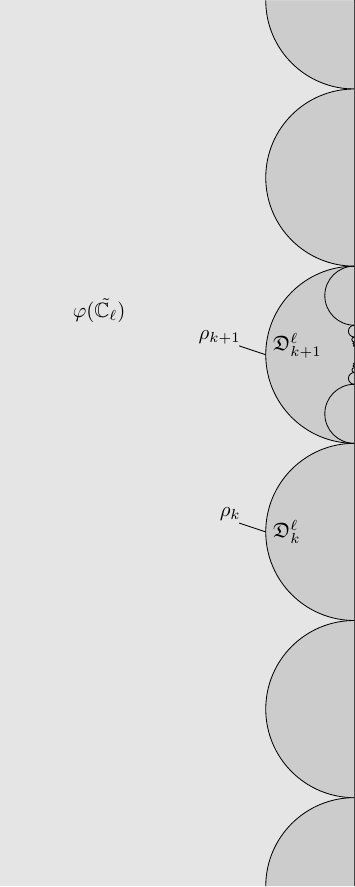}
\end{center}
\caption{A conformal image  of the Riemann surface of  generation $2$     }\label{Fig3-braids3.pdf}
\end{figure}


The supremum of the diameters of the half-circles of generation $n$ does not exceed $2^{-n}$. Indeed, the half-circles of generation $n$ are obtained as follows. Take a half-circle $\rho$ of generation $n-1$ and reflect all other half-circles of generation $n-1$ across it. Each half-circle of generation $n-1$ different from $\rho$ has diameter on one side of the diameter of  $\rho$.
Hence, the image of each half-circle of generation
$n-1$ different from $\rho$ under reflection across $\rho$ is contained in a quarter disc, hence has diameter not exceeding half of the diameter of $\rho$. The statement is obtained by induction. Since the supremum of the diameters of the half-circles of generation $n$  tends to zero for $n\to \infty$, we obtain in the limit a conformal mapping $\varphi$ from
$U$ onto ${\mathbb{C}}_{\ell}$.  \hfill $\Box$

Consider the inverse of the mapping $\varphi_0: \widetilde{\mathbb{C}}_{\ell}\to\mathbb{C}_{\ell}
\setminus
\bigcup _{k=-\infty}^{\infty} \overline{\mathfrak{D}^{\ell}_k}$. By construction the composition ${\sf P}\circ \varphi_0^{-1}$ of $\sf P$ with its inverse has the following property: the restriction to $ (\mathfrak{H}_0+ik)\setminus (\mathfrak{D}_0+ik)$ of
${\sf P}\circ \varphi_0^{-1}:\mathbb{C}_{\ell}
\setminus
\bigcup _{k=-\infty}^{\infty} \overline{\mathfrak{D}^{\ell}_k}\to{\mathbb{C}}_{\ell}$  is a conformal mapping onto $\mathfrak{H}_0+ik$ that takes $ik$ to $ik$, $i(k+1)$ to $i(k+1)$, and $\infty$ to $\infty$. Hence,  ${\sf P}\circ \varphi_0^{-1}$ commutes with translation by $i$, i.e.
\begin{align}\label{eq3braids.102}
{\sf P}\circ \varphi_0^{-1}(z+i)= {\sf P}\circ \varphi_0^{-1}(z)+i \;\;\mbox{ for} \;\;z\in \mathbb{C}_{\ell}
\setminus
\bigcup _{k=-\infty}^{\infty} \overline{\mathfrak{D}^{\ell}_k}\,.
\end{align}

\section{Building blocks and their extremal length}\label{sec:3-braids2b}

Consider an elementary slalom curve in the left half-plane with initial
point in the interval $(ik, i(k + 1))$ and with terminal point in the interval $(i j,i(j + 1))$.
Suppose that
$|k-j| \geq 2$, i.e.  the slalom curve is non-trivial.
After a
translation by a (half-integer or integer) multiple of $i$ we
obtain  a
curve
that has endpoints in the intervals $( -i(M +1),-
iM)$ and $(iM, i(M +1))$ with $M = \frac{|k-j|-1}{2}\geq
\frac{1}{2}$.
The homotopy class of the  slalom curve has the same extremal length as the translated class.

Let $M$ be any positive number.
Denote a curve with interior in the open left half-plane, with initial point in $( -i(M +1),-
iM)$ and terminal point in $(iM, i(M +1))$, by $\gamma_{\ell,M}$, and the curve with inverted orientation by $\gamma_{\ell,M}^{-1}$.
Notice that with the afore mentioned value $M = \frac{|k-j|-1}{2}$ the curve $\gamma_{\ell,M}$ is not an elementary slalom curve for
even $|k-j|$, but the curve $\gamma_{\ell,M} + iM$ is always an elementary
slalom curve. It will
be convenient to work  with the normalized curve $\gamma_{\ell,M}$
for computing the extremal length.
\index{$\gamma_{\ell,M}$} \index{$\gamma_{\ell,M}^{-1}$} \index{$\gamma^*_{\ell,M}$} \index{$\gamma_{r,M}$}
\index{$\gamma^*_{r,M}$}

For any integer or half-integer number $M\geq \frac{1}{2}$ we assign to the curve $\gamma_{\ell,M}$ the
class
$\gamma^*_{\ell,M}$ of curves which are homotopic to
$\gamma_{\ell,M}$
through curves in $\mathbb{C} \setminus i(\mathbb{Z}-M)$ with initial
point
in
$(-i(M + 1), -iM)$ and terminal point in $(iM, i(M+ 1))$.
The class assigned to $\gamma_{\ell,M}^{-1}$ in this way is denoted
by $(\gamma_{\ell,M}^*)^{-1}$.
Respectively, let  $\gamma_{r,M}$ be a curve with
interior  contained
in the right half-plane ${\mathbb{C}_{r}}$ with initial point  in the
interval
$(-i(M+1),-iM)$ and terminal point in $(iM,i(M+1)$. Consider the
class
$\gamma^*_{r,M}$ of curves which are homotopic to
$\gamma_{r,M}$
through curves in $\mathbb{C} \setminus i(\mathbb{Z}-M)$ with initial
point
in
$(-i(M + 1), -iM)$ and terminal point in $(iM, i(M+ 1))$.

We will call a homotopy class of elementary slalom curves in $\mathbb{C} \setminus i\mathbb{Z}$ an elementary slalom class with parameter $M$,
if after a translation by an imaginary integer or half-integer it becomes equal to
$(\gamma_{\ell,M}^*)^{\pm 1}$ or $(\gamma_{r,M}^*)^{\pm 1}$ .
A homotopy class  of elementary half-slalom curves in $\mathbb{C} \setminus i\mathbb{Z}$ with endpoints in  a real line $\frac{i}{2}+ i k+\mathbb{R}$ ($k\in \mathbb{Z}$) and a component of $i\mathbb{R}\setminus i\mathbb{Z}$
is called an elementary half-slalom class with parameter $M$,
if for a representative $\gamma$ the union $\bigcup_{\gamma\in \Gamma}(\gamma\cup \gamma_{\rm refl})$ represents an elementary slalom class with parameter $M$.
Here $\gamma_{\rm refl}$ is the reflection of $\gamma$ through the real line $\frac{i}{2}+ik +\mathbb{R}$. Notice that for a given $M$ the extremal lengths of all four classes
$(\gamma_{\ell,M}^*)^{\pm 1}$ or $(\gamma_{r,M}^*)^{\pm 1}$ are equal. The same remark concerns the respective half-slalom classes.
\index{elementary slalom class with parameter $M$}

\begin{lemm}\label{lemxx}
The extremal length of an elementary slalom class $\gamma^*$ with parameter $M$ equals
\begin{align}\label{eq6}
\Lambda(\gamma^*)&= \frac{4}{\pi}\log(\sqrt{M}+\sqrt{M+1})\,.
\end{align}
The extremal length of an elementary half-slalom class $\tilde\gamma^*$ with parameter $M$ equals
\begin{align}\label{eq6'}
\Lambda(\tilde\gamma^*)&= \frac{2}{\pi}\log(\sqrt{M}+\sqrt{M+1})\,.
\end{align}
\end{lemm}
\noindent {\bf Proof.}$\,$  It is enough to prove that
$\;\Lambda(\gamma^*_{\ell,M})=\frac{4}{\pi}\log(\sqrt{M}+\sqrt{M+1})
\,.\;\,$
Let $\,$ $\varphi:U\to \mathbb{C}_{\ell}$ be the mapping considered in the previous section.
For a rectangle $R$ we take a holomorphic mapping $f:R\to \mathbb{C}\setminus i\mathbb{Z}$,
that represents $\gamma^*_{\ell,M}+iM$. Consider its lift $\tilde f:R\to U$ to a mapping to the universal covering $U$ of $\mathbb{C}\setminus i\mathbb{Z}$. Define the mapping $G(\zeta)= -M-i\varphi(\zeta),\; \zeta \in U,$ from $U$ onto the upper half-plane.
The composition $G\circ \tilde{f} $
maps $R$ into $\mathbb{C}_+$, and its extension to the closure of $R$ takes the horizontal sides to the circles $\{|z\pm (M+\frac{1}{2})|=\frac{1}{2}\}$.
Lemmas \ref{lemm5} and  \ref{lem4.1'} imply the inequality $\lambda(R)\geq \frac{4}{\pi}\log(\sqrt{M}+\sqrt{M+1})$.
(Compare also with the proof of the statement in Example 1 of Section \ref{sec:4.1b}.)

Take the conformal mapping $f_0$ of a rectangle $R_0$ with $\lambda(R_0)= \frac{4}{\pi}\log(\sqrt{M}+\sqrt{M+1})$ onto $\mathbb{C}_+\setminus \{|z\pm  (M+\frac{1}{2})|\leq\frac{1}{2}\}$ that maps the horizontal sides to the two boundary circles $\{|z\pm  (M+\frac{1}{2})|=\frac{1}{2}\}$. Precompose this mapping with the inverse of $G$, and project the obtained mapping to
$\mathbb{C}\setminus i\mathbb{Z}$. We obtain a mapping that represents $\gamma^*_{\ell,M}+iM$. Hence,  $\Lambda(\gamma^*_{\ell,M})=     \Lambda(\gamma^*_{\ell,M}+iM)= \frac{4}{\pi}\log(\sqrt{M}+\sqrt{M+1})$.

The statement for the half-slalom class $\tilde{\gamma}^*$ follows in the same way.
\hfill $\Box$

\begin{lemm}\label{cor2}  The extremal length of an element of each of
the relative fundamental groups
$\pi_1^{pb}$, $\pi_1^{tr}$, $^{tr}\pi_1^{pb}$, and  $^{pb}\pi_1^{tr}$
is
realized on a locally conformal mapping of a rectangle representing the
element. The extremal mapping extends locally conformally across the open
horizontal sides of the rectangle.

The extremal length of a conjugacy class of elements of the
fundamental group of the twice punctured plane is realized on a locally
conformal mapping of an annulus into the twice punctured plane.
\end{lemm}

\noindent {\bf Proof of Lemma \ref{cor2}.} For this proof it is more convenient to consider the fundamental group of the complex plane punctured at $0$ and $1$ rather than at $-1$ and $1$, and to consider the upper half plane $\mathbb{C}_+$ as universal covering of $\mathbb{C} \setminus \{0,1\}$.

Each holomorphic mapping of a
rectangle into $\mathbb{C} \setminus \{0,1\}$ that represents an element of one of the relative fundamental groups lifts to the universal covering. The lift
takes the
open horizontal sides of the rectangle to certain half-circles with
diameter on
the real axis (maybe, after a conformal self-map of the half-plane)
and
represents the class of curves that are contained in the half-plane and join the two
half-circles.
As in the proof of Lemma \ref{lem4.1'} the extremal length is realized on a conformal
mapping of a rectangle
onto the half-plane with two deleted half-discs such that the horizontal sides are mapped onto the half-circles. Composing with the
covering map
we obtain a locally conformal mapping that extends locally conformally across the open horizontal sides of the rectangle.

We will now prove the statement concerning conjugacy classes of elements of the fundamental group of $\mathbb{C} \setminus
\{0,1\}$ with base point.

Recall that each element of the fundamental group corresponds to a covering
transformation.  Each covering transformation is a holomorphic
self-homeomorphism of the universal covering $\mathbb{C}_+$  that
extends to a M\"obius transformation $T(z)=\frac{az+b}{cz+d}$ of
the Riemann sphere with integer coefficients $a,\,b,\,c,\,d\, ,$ such that
$ad-bc=1$. Moreover, $T$ is either parabolic (i.e. $T$ has one fixed
point and is conjugate to the mapping $z \to z+b'$ for a constant
$b'$), or $T$ is hyperbolic (i.e., T has two fixed points and is
conjugate to $z \to\kappa z$ for a positive real number $\kappa$), or
$T$ is elliptic (i.e., T has two complex fixed points symmetric with
respect to the imaginary axis and is conjugate to $z \to e^{i\theta}z$
for a real number $\theta$). See \cite{Lehn}, Chapter II, 9D and 9E.

Let $\hat a$ be a conjugacy class of
elements of the fundamental group of $\mathbb{C} \setminus
\{0,1\}$ and let $a$ be an element of the fundamental group
that represents $\hat a$. Denote by $T_a$
the covering transformation corresponding to $a$, and by $\langle
T_a\rangle$ the subgroup of the group of covering transformations
generated by $T_a$. Then the quotient $\mathbb{C}_+ \diagup
\langle T_a\rangle$ is an annulus. It has extremal length $0$ if $T_a$
is parabolic or elliptic and has positive extremal length if $T_a$ is
hyperbolic. If $f:A \to \mathbb{C} \setminus\{0,1\}$ is a holomorphic
mapping of an annulus $A$ to  $\mathbb{C} \setminus\{0,1\}$ that
represents $\hat a$, then $f$ lifts to a holomorphic map of $A$ into
$\mathbb{C}_+ \diagup \langle T_a\rangle$. The lift represents the
class of  a generator of the fundamental group of $\mathbb{C}_+
\diagup \langle T_a\rangle$. The corollary follows from Lemma \ref{lemm6}.
\hfill $\Box$

\medskip

The following proposition considers elements of the fundamental group
$\pi_1$, that are represented by elementary words (i.e. by words consisting of a single syllable), and gives upper and lower bounds for
the
extremal length of the corresponding classes of curves in the relative
fundamental groups with $pb$, $tr$ or mixed boundary values.

\begin{prop}\label{prop4b} The following statements hold for
elements of the relative fundamental groups that are represented by elementary words.
\begin{itemize}
\item[$1$.] Each syllable $\mathfrak{s}$ of form $(1)$ and degree $d \geq2$ with $pb$ horizontal boundary values lifts under $f_1\circ f_2$ to an elementary slalom class with parameter $M=\frac{d-1}{2}$.
The following equality for the extremal length holds
\begin{align*}
\Lambda(_{pb}(\mathfrak{s})_{pb})= \frac{2}{\pi} \log(d+\sqrt{d^2-1})\,.
\end{align*}
\item [$2$.] Each syllable $\mathfrak{s}$ of form $(2)$ and degree $d\geq 2$ with $tr$ horizontal boundary values lifts under $f_1\circ{\tilde f}_2$ to an elementary slalom class with parameter $M=\frac{d-1}{2}$.
    The following equality for the extremal length holds
    \begin{align*}
\Lambda(_{tr}(\mathfrak{s})_{tr})= \frac{2}{\pi} \log(d+\sqrt{d^2-1})\,.
\end{align*}
\item [$3$.] Any syllable $\mathfrak{s}$ of degree $d\geq 1$ with mixed boundary values lifts under $f_1\circ f_2$ to an elementary half-slalom class with parameter $M=d-\frac{1}{2}$. For the extremal length the equalities
    \begin{align*}
\Lambda(_{pb}(\mathfrak{s})_{tr})= &\frac{1}{\pi} \log(2d+\sqrt{4d^2-1})\,,\\
\Lambda(_{tr}(\mathfrak{s})_{pb})= &\frac{1}{\pi} \log(2d+\sqrt{4d^2-1})
\end{align*}
hold.
\end{itemize}
\end{prop}
\noindent {\bf Proof of Statement 1.}
Assume first that  $n=d>0$. For any integer number $k$
the class $_{pb}(a_1^n)_{pb}$ lifts
under $f_1 \circ f_2$ to an elementary homotopy slalom class with
initial point in ($i(k-1),ik)$ and terminal point in $i(k-1+d), i(k+d))$. There is a representative in the closed left half-plane.
The slalom class has parameter $M=\frac{d-1}{2}$.
A lift of  the class $_{pb}(a_1^{-d})_{pb}$ is obtained by inverting the orientation.
\begin{figure}[H]
\begin{center}
\includegraphics[width=11cm]{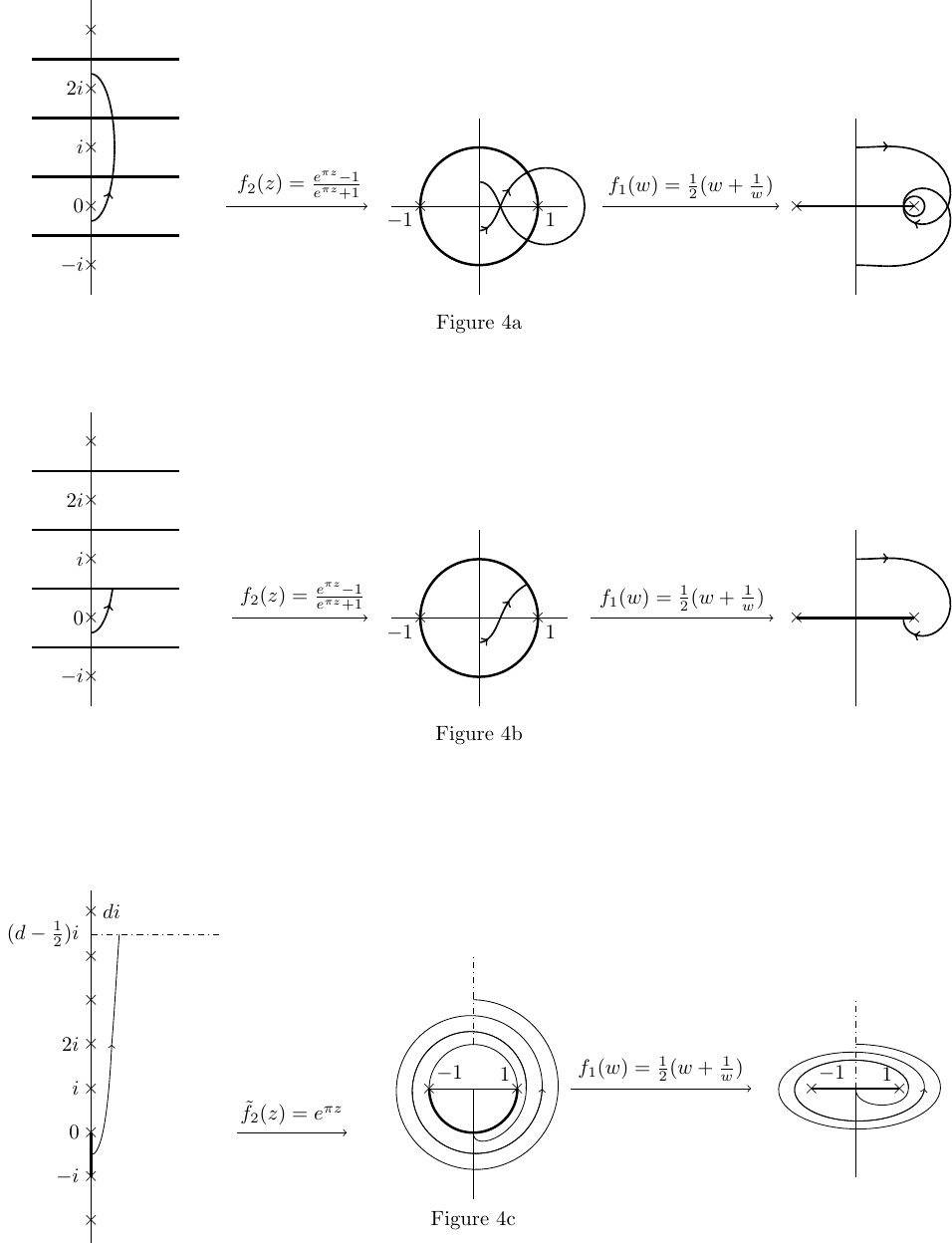}
\end{center}
\caption{Some elementary slalom curves and their projections under different coverings}\label{Fig3-braids4.pdf}
\end{figure}
For $n=d$ the class $_{pb}(a_2^n)_{pb}$ lifts under $f_1\circ f_2$ to an elementary homotopy slalom class, represented by a curve in the closed right half-plane with
initial point in $\big(i(k-1), ik\big)$ and terminal point in $\big(i(k-1-d),i(k-d)\big)$. The class has parameter $M=\frac{d-1}{2}$.
A lift of the class $_{pb}(a_2^{-d})_{pb}$ is obtained by inverting the orientation.
See Figure 4a for an elementary slalom
curve with $d=3$ that is the lift under $f_1 \circ f_2$ of a
representative of $_{pb}(a_2^{-3})_{pb}$ with initial point not equal
to $0$.

The estimate for the extremal length follows
from formula \eqref{eq6}, since
\begin{align}\label{eq3braids.100}
\frac{2}{\pi}\log(\sqrt{M}+\sqrt{M+1})^2&=\frac{2}{\pi}\log(M+M+1+2\sqrt{M(M+1)})\nonumber\\
=\frac{2}{\pi}\log(d+2\sqrt{\frac{d-1}{2} \cdot\frac{d+1}{2}})&=\frac{2}{\pi}\log(d+\sqrt{d^2-1})\,.
\end{align}


\noindent {\bf Proof of Statement 3 for syllables of the form $a_j^n$, $|n|\geq 1$, with mixed boundary values.} Let $n=d>0$. For each $k\in \mathbb{Z}$ the class $_{pb}(a_1^n)_{tr} \in\,^{pb}\pi_1^{tr}$
lifts under the covering $f_1 \circ f_2$ to a class of elementary
homotopy half-slalom curves with  initial point in
$(ik,i(k+1))$ and terminal point in
$i(k+d+\frac{1}{2})+\mathbb{R}$.  The half-slalom class has parameter $M=d-\frac{1}{2}$, and can be represented by a curve contained in the closed left half-plane.
In the remaining cases ($n<0$, or $a_1$ replaced
by $a_2$, or both) a lift of the curve can be obtained
from the present one
by suitable choices of inverting orientation or
reflection in the imaginary line or both.
See Figure 4b for $_{pb}(a_2^{-1})_{tr}$ with
$d=1$.

The estimate for the extremal length follows by Lemma \ref{lemxx},
\eqref{eq6'}.

The case $_{tr}(a_j^n)_{pb}$ is treated symmetrically.

\medskip

\noindent {\bf Proof of Statement 3 for syllables of form (2) with mixed boundary values.}
Recall that  for each integer $k$ the mapping $f_1 \circ \tilde f_2$ takes $(ik,i(k+1))$ onto $(-1,1)$, and $\frac{i}{2} +ki +\mathbb{R}$ onto $i(0,\infty)$, if $k$ is even and onto  $-i(0,\infty)$, if $k$ is odd.
Hence, for each $k\in \mathbb{Z}$ the class
$_{tr}(\mathfrak{s})_{pb}=\,_{tr}(a_1 \,
a_2 \ldots)_{pb}$ lifts under $f_1 \circ \tilde f_2$  to an elementary homotopy half-slalom class
in the closed left half-plane
with initial point in $((2k-1)i,2ki)$
and terminal point in
$i(2k-d-\frac{1}{2}) + \mathbb{R}$, and it lifts also under $f_1 \circ \tilde f_2$  to an elementary homotopy half-slalom class
in the closed right half-plane
with initial point in $((2ki,(2k+1)i)$
and terminal point in
$i(2k+d+\frac{1}{2}) + \mathbb{R}$.
Here $d$ is the number of letters in
$\mathfrak{s}$. See also
Proposition \ref{prop5a}, as well as Figure \ref{Fig3-braids1}.
The class has parameter $M= d-\frac{1}{2}$.

For each $k\in \mathbb{Z}$ the class $_{tr}(\mathfrak{s})_{pb}=\,_{tr}(a_2 \,
a_1 \ldots)_{pb}$ lifts under $f_1 \circ \tilde f_2$  to an elementary homotopy half-slalom class
in the closed right half-plane with initial point in
$ ((2k-1)i,2ki)$ and
with terminal point in
$ i(2k+d-\frac{1}{2})+\mathbb{R}$, and it also lifts under $f_1 \circ \tilde f_2$  to an elementary homotopy half-slalom class
in the closed left half-plane with initial point in
$ (2ki,(2k+1)i)$ and
with terminal point in
$ i(2k-d+\frac{1}{2})+\mathbb{R}$.
Again $d$ is the
degree of the elementary word.
See Figure 4c for $\,_{tr}(a_2 \,
a_1 a_2 a_1 a_2)_{pb}$ with number of half-turns equal to $d=5$.


If all generators enter with power $-1$ the orientation is reversed.
In all cases we obtain elementary half-slalom curves with parameter $M=d-\frac{1}{2}$. The estimate of the extremal length follows from
\eqref{eq6'}.

\noindent{\bf Proof of Statement 2 (Syllables of form (2) with $tr$ boundary values).} The proof of Statement 2 is related to the proof of Statement 1 in the same way
as the proof of Statement 3 for syllables of form (2) is related to the proof of Statement 3 for syllables of form (1). We
leave it to the reader.

The proposition is proved. \hfill $\Box$

\medskip

\section{The extremal length of words in $\pi_1$. The lower bound}\label{sec:3-braids2}
Take an element of a relative fundamental group
of the twice punctured complex plane.
We will break representing curves into elementary pieces. The pieces will be
chosen so that we have a good lower bound of the extremal length of the
homotopy class of each piece. Ahlfors' Theorem B and Corollary \ref{cor4.0} will give a lower
bound for the extremal length of the element of the relative
fundamental group by the sum of the extremal lengths of the classes of
the elementary pieces.

We will use the following terminology.
Let $R$ be a rectangle. By a curvilinear subrectangle $R'$ of $R$ we mean a simply
connected domain in $R$ whose boundary looks as follows. It consist of two
vertical segments, one in each vertical side of $R$, and either two disjoint
simple arcs with interior contained in $R$ and endpoints on
opposite open vertical sides of $R$, or one such arc and a horizontal side
of the rectangle $R$. The rectangle $R$ itself may also be considered as a
curvilinear subrectangle of $R$. The vertical segments in the boundary of the curvilinear rectangle $R'$ are
considered as its vertical curvilinear sides, the remaining curvilinear sides are considered as the horizontal curvilinear sides of $R'$.
Recall that each curvilinear rectangle admits a conformal mapping onto a true rectangle whose
continuous  extension to the boundary takes curvilinear horizontal sides to horizontal sides and curvilinear vertical sides to vertical sides.
\index{curvilinear subrectangle}

We will use the symbol $\#$ for the boundary values if we are
free to choose either $pb$ or $tr$ boundary values.
Let $v_1$ and $v_2$ be words in $\pi_1$. For the word $_{\#}(v)_{\#} =
\,_{\#}(v_1 \,v_2)_{\#} \in \,^{\#}\pi_1\,^{\#}$ we say
that there is a sign change of exponents for the pair $(v_1,v_2)$ if
the sign of
the exponent of the last term of $v_1$ is different from the sign of
the
exponent of the first term of $v_2$.

\medskip

The following lemma is a key part for the proof of the lower bound of the
extremal length of slalom classes, respectively, of elements of the relative fundamental group of $\mathbb{C}\setminus \{-1,1\}$.

\begin{lemm}\label{lemm15}
Consider a reduced word  $w \in \pi_1$.
Take a locally conformal mapping $g$ from a neighbourhood of the closure $\bar R$ of a rectangle $R$ into the twice punctured plane $\mathbb{C} \setminus\{-1,1\}$, such that the restriction $g|R$
represents the word $w$  with $pb$, $tr$ or mixed
boundary values.

The following statements hold.
\begin{itemize}
\item[1.]  If $w$ has at least two terms, then $R$ can be divided into a collection of pairwise disjoint
curvilinear subrectangles $R_j$ of $R$ which are in bijection to the
terms of $w$ in such a way
that the restriction of g to $R_j$ represents the $j$-th
term with $pb$ boundary values if the term is not at
the (right or left) end of the word and, possibly, with mixed
boundary values if the term is at the end of the word.
\item [2.]
If $w$ has at least two syllables and among them there is at least one syllable of form $(1)$ then $R$
contains a collection of mutually disjoint curvilinear subrectangles
$R_{j}'$  of $R$, $ j \in J,\,$
which are in bijective correspondence to the collection of
syllables $\mathfrak{s}_{j}$ of $w$ that are not of the form
$(1)$ such that either the
restriction of $g$ to $R_{j}'$ represents the respective
syllable $\mathfrak{s}_{j}$ with mixed boundary values,
or the restriction represents an elementary
word of form $(2)$ with one more letter than
$\mathfrak{s}_{j}$ and mixed
boundary values.

\item[3.] If the word has at least two syllables and all syllables $\mathfrak{s}_j,\, j=1,\ldots,N,$ of the word $w$  are of form $(2)$ or $(3)$ then there is a
division of $R$ into curvilinear subrectangles $R_j'$ of $R$, $ j=1,\ldots,N,\,$
such that
for $j=1,\ldots,N-1,\,$ the restriction $g\mid R_j'$  represents the
$j$-th
syllable with mixed boundary values. In the same way there is a
decomposition
of $R$ into curvilinear subrectangles $R_j''$ of $R$, $ j=1,\ldots,N,$ such
that for
$j=2,\ldots,N,$ the restriction $g\mid R_j''$  represents the $j$-th
syllable
with mixed boundary values.
\end{itemize}
\end{lemm}

For the proof of Statements 2 and 3 of Lemma \ref{lemm15} and we need Statements 2 and 3 of the following lemma.
\begin{lemm}\label{lemm15'}
\begin{itemize}
\item[ $1.$] Let $\hat{w}$ be a conjugacy class of elements of $\pi_1$, $\hat{w}\neq\widehat{a_j^k}$ for $j=1$ or $j=2$ and integers $k$. Take a cyclically reduced word $w$ that represents the conjugacy class $\hat w$, and a mapping $g:A\to \mathbb{C}\setminus \{-1,1\}$ of an annulus $A$ to $ \mathbb{C}\setminus \{-1,1\}$, that extends to a locally conformal mapping on a neighbourhood of the closure $\bar A$
and represents $\hat w$. Then there exists a smooth arc $L^0\subset \bar A$ with endpoints on different boundary components of $A$, such that $g\mid A\setminus L^0$ represents $_{pb}w_{pb}$.
\item[ $2.$] Suppose $g:R\to \mathbb{C}\setminus\{-1,1\}$ extends to a locally conformal mapping in a neighbourhood of the closure $\bar R$ of the rectangle $R$ and represents  $_{\#}({a_{j_1}^{n_1}\,a_{j_2}^{n_2}})_{\#}$, where $a_{j_k}$ are different standard generators of $\pi_1$, and the integer numbers $n_1$ and $n_2$ have different sign. Then
there exists a smooth arc $L^0\subset \bar R$ with endpoints on different open vertical sides of $R$, such that
$R\setminus L^0=R_1\cup R_2$ for two disjoint open curvilinear rectangles,
the restriction $g\mid R_1$ represents  $_{\#}(a_{j_1}^{n_1})_{tr}$, and the restriction $g\mid R_2$ represents  $_{tr}(a_{j_2}^{n_2})_{\#}$.
\item[ $3.$] Suppose $g:R\to \mathbb{C}\setminus\{-1,1\}$ extends to a locally conformal mapping in a neighbourhood of the closure $\bar R$ of the rectangle $R$ and represents   $_{\#}(a_{j}^{n})_{\#}$ for a standard generator $a_j$ and an
integer number $n$ with $|n|>1$. Let $n_1$ and $n_2$ be non-zero integer numbers of the same sign, such that $n=n_1+n_2$. Then there is  a smooth arc $L^0\subset \bar R$ with endpoints on different open vertical sides of $R$, such that
$R\setminus L^0=R_1\cup R_2$ for two disjoint open curvilinear rectangles, $g\mid R_1$ represents  $_{\#}(a_{j}^{n_1})_{tr}$ and $g\mid R_2$ represents  $_{tr}(a_{j}^{n_2})_{\#}$.
\end{itemize}
 \end{lemm}
We prove Lemma \ref{lemm15'} after Lemma  \ref{lemm15} is proved.
\medskip

\noindent {\bf Proof of Lemma \ref{lemm15}}. We start with the {\bf proof of Statement 1.}
Notice first that for the reduced word $w=w_1^{n_1} \,
w_2^{n_2} \ldots
w_k^{n_k}\in \pi_1$ the corresponding element  $_{\#}(w)_{\#}$
can be represented by a smooth curve $\beta'$ that intersects $\mathbb{R}$
and $i\mathbb{R}$ transversally and lifts under $f_1\circ f_2$ to a slalom curve (not merely to a homotopy slalom curve) in the case of $pb$ horizontal boundary values, or to the union of a slalom curve with some elementary half-slalom curve(s) if some horizontal boundary values of $w$ are $tr$. The intersection points of $\beta'$ with the imaginary axis divide $\beta'$ into pieces $\beta'_j$ that represent $w_j^{n_j}$, $j=1,\ldots k$. The $\beta'_j$ represent elements $_{\#}(w_1^{n_1})_{pb}$ if $j=1$, $_{pb}(w_j^{n_j})_{pb}$ for $j=2,\ldots,k-1,$ and $_{pb}(w_k^{n_k})_{\#}$ if $j=k$.

Consider the mapping $g$ of the lemma.
By our condition
$0$ is a regular value of both, $\mbox{Re}\,g(z)$ and $\mbox{Im}\,g(z)$  in a neighbourhood of the closed rectangle $\bar R$.
Normalize the rectangle so that we have $R=\{z \in \mathbb{C}: x \in (0,1),\, y \in (0,\sf{a})\}$. Consider the mentioned curve $\beta'$.

Take a homotopy that joins the mentioned curve $\beta'$ with the restriction $\beta^0$ of $g$ to the left vertical side of $R$, i.e. we consider a smooth mapping $h$ from $\overline{R'}\stackrel{def}= \{z \in \mathbb{C}: x \in [-1,0],\, y \in [0,\sf{a}]\}$ to $\mathbb{C}\setminus\{-1,1\}$, whose restriction to the left side of $R'$ equals $\beta'$ and whose restriction to the right side of $R'$ equals $\beta_0$.
Since $\beta'$ and $\beta^0$ are smooth
curves that intersect $(-1,1) \cup i\mathbb{R}$ transversely we may assume that zero is a regular value of both, $\mbox{Re}\,h$ and $\mbox{Im}\,h$. Indeed, there is a
neighbourhood $U$
of the two closed vertical sides of the rectangle
such that $0$ is a regular value of both, $\mbox{Re}\,h$ and $\mbox{Im}\,h$,  on $U$, and
we may assume by choosing the homotopy and extending it to aneighbourhood of $\overline{R'}$  suitably that
in a neighbourhood $U'$ in $\overline{R'}$ of the closed horizontal sides this is so.
Let $ \Psi$ be a non-negative smooth function in a neighbourhood of the closed
rectangle $\overline{R'}$ which equals $1$ outside $U \cup U'$ and equals zero
near the boundary of the rectangle. If $\varepsilon$ is a sufficiently small complex value such that $\mbox{Re}\,\varepsilon $ is a regular value  of $\mbox{Re}\,h$ on the closed rectangle, and $\mbox{Im}\,\varepsilon $ is a regular value of
$\mbox{Im}\,h$ on the closed rectangle,
then $h-\varepsilon \Psi$ is another
smooth homotopy joining $\beta^0$ with $\beta^1$, and
zero is
a regular value of both, the real part and the imaginary part of this mapping
on the closed rectangle.
We may choose the homotopy $h$ so that we obtain a smooth mapping  $\tilde g$ from a neighbourhood of the closed rectangle
$\overline{ R'} \cup \bar R$ to $\mathbb{C}\setminus\{-1,1\}$,
which equals $g$ on $\bar R$ and $h$ on $\overline{R'}$ such that $0$ is a regular value of both, $\mbox{Re}\,{\tilde g}$ and $\mbox{Im}\,{\tilde g}$. We defined $ {\tilde  R}$ as the interior of $\overline{ R'} \cup \bar R$.

Each connected component of the level set $\{L_0 \stackrel{def}= z \in \tilde{ R}: \mbox{Re} \,h(z)= 0\}$ on the rectangle $\tilde{ R}$ is
either an open arc (i.e. it is non-compact), and both end points of its closure are on
the boundary of $\tilde{ R}$, or it is a circle (i.e. it is compact). Consider the closures of all non-compact  connected components of $L_0$. This is a collection of closed arcs. The endpoints of these arcs that are contained in
the left side of the rectangle divide the left side into connected
components. The restrictions of $h$ to these components
are the $\beta'_j$.
The closure of each non-compact component of $L_0$ has at most one endpoint on the open left side of the rectangle.
Indeed, assume the contrary. Then the restriction of $\beta'$ to the interval $I$ on the left side of the
rectangle between the two endpoints of the arc is equal to the product of at least one or more successive curves $_{pb}(\beta'_{j}) _{pb}$ and, hence is not homotopic to a constant curve through curves with $pb$ boundary values.
On the other hand, the existence of an arc
in the level set $\mbox{Re}\,h= 0$ joining the two endpoints of $I$
would
provide a relative homotopy in $\mathbb{C}\setminus \{-1,1\}$
(with endpoints in the imaginary axis) joining the restriction
of $h$ to the interval $I$ with a constant curve contained in the imaginary
axis, which is a contradiction.

The same reasoning shows that the closure of each non-compact component of $L_0$
with
one endpoint on the open left side of the rectangle cannot have its other endpoint on a closed horizontal side of $R$.
Indeed, otherwise the restriction of $h$ to
the interval on the vertical side between the endpoint of the arc and the vertex of the rectangle belonging to the closed horizontal side would be homotopic to a constant through curves with endpoints in $(-1,1) \cup i\mathbb{R}$, which is impossible.

We may ignore the non-compact components of $L_0$ whose closures have no endpoint on the open left side of $R$,
and will also ignore the circles contained in $L_0$.
We consider the connected components of $L_0$ whose closures have one
endpoint on the open left side of $R$,
and the other endpoint on the open right
side. These arcs divide the rectangle into
curvilinear rectangles which are in bijective correspondence
to the intervals of division on the left side. We call them dividing arcs. Take the
curvilinear rectangle whose left side corresponds to
$\beta'_j$. The restriction of $h$ to this curvilinear rectangle
provides a homotopy with boundary values in
the imaginary axes (in $(-1,1) \cup i\mathbb{R}$, respectively), that joins $\beta'_j$ to the restriction $\beta^1_j$ of $h$ to the right side
of the curvilinear rectangle $\tilde{R}$.
Since the division of the rectangle  $\tilde{R}$ into
curvilinear rectangles induces a division of the right side of
the rectangle into intervals, the curve $\beta^1$, which is the restriction of $g$ to the right side of $R$, is the
composition of the curves
$\beta^1_j$. Denote by $E$
the subset of the right side of $R$ that consists of the endpoints of the mentioned intervals.
\index{arc ! dividing}

Consider all connected components of the level set $\{z \in  R: \mbox{Re} \, {g}=0\}$ in the original rectangle $R$ that have an endpoint on the right side of $R$ contained in $E$. Each such component is a part of a dividing arc for $\tilde R$. Hence, each such component has its other endpoint on the left side of the original rectangle $R$. We call these components the dividing arcs for $R$. The dividing arcs for $R$ provide a division of the rectangle $R$ into curvilinear rectangles $R_j$. The right side of each $R_j$ is the connected component $I_j$ of the complement of $E$ in the right side of $R$ and the restriction of $g$ to $I_j$  equals $\beta^1_j$. Since $\beta^1_j$ represents the term $w_j^{n_j}$  with the required horizontal boundary conditions, we obtained the required collection of curvilinear rectangles $R_j$.  Statement 1 is proved.

\medskip

\noindent {\bf Proof of Statement 2.} Choose a syllable $\mathfrak{s}_k$ of form (1).
Suppose there are syllables on the left of $\mathfrak{s}_k$ which are not of form (1). Consider them in the order from left to right.

{\bf If the left boundary values of $w$ are $pb$} then
we consider the most left syllable $\mathfrak{s}_{j_1}$ with $j_1< k$
which is not of form (1). 
For the curvilinear rectangle  $R_{j_1}$ associated to  $\mathfrak{s}_{j_1}$ the restriction $g\mid R_{j_1}$ has $pb$ left boundary values.

{\bf Suppose the next syllable $\mathfrak{s}_{j_1+1}$ to the right of
$\mathfrak{s}_{j_1}$ is not of form (1).} Then  $j_1+1 <k$ and there is a sign change of exponents for the pair $(\mathfrak{s}_{j_1},\mathfrak{s}_{j_1+1})$ of consecutive syllables.
For the curvilinear rectangles $R_j$ of Statement 1 the restriction of
$g$ to $R_{j_1,j_1+1} \stackrel{def}{=} \mbox{Int}(\overline R_{j_1} \cup \overline
R_{j_1+1})$ represents $_{pb}(\mathfrak{s}_{j_1} \mathfrak{s}_{j_1+1})_{pb}$.
(Recall that $ \mbox{Int}X$ denotes the interior of a subset $X$ of a topological space.)
By Lemma \ref{lemm15'}, Statement 2
the curvilinear rectangle $R_{j_1,j_1+1}$ can be divided into two curvilinear rectangles such that the restriction of $g$ to
them represents the syllables
$_{pb}(\mathfrak{s}_{j_1})_{tr}$ and
$_{tr}(\mathfrak{s}_{j_1+1})_{pb}$. We obtained the required representation of the two syllables
$\mathfrak{s}_{j_1}$ and $\mathfrak{s}_{j_1+1}$ which are both not of
form (1).

{\bf Suppose the next syllable $\mathfrak{s}_{j_1+1}$ to the right of
$\mathfrak{s}_{j_1}$ is of
form (1)}, i.e
it equals $w_{j_1+1}^n$ for an integer $n\geq 2$ with $w_{j_1+1}$ being
a standard generator of $ \pi_1$. We include the case when $j_1+1=k$.
If there
is a sign change of exponents for the pair $(\mathfrak{s}_{j_1},w_{j_1+1}^n)$
the preceding argument applies.
Suppose there is no sign change. By Lemma \ref{lemm15'} Statement 3
the curvilinear rectangle $R_{j_1,j_1+1}$ can be divided into
two curvilinear rectangles such that the restriction of $g$ to them
represents the syllables
$_{pb}(\mathfrak{s}_{j_1}\, w_{j_1+1}^{{\rm{sgn}}(n)})_{tr}$
and $_{tr}(w_{j_1+1}^{n-{{{\rm{sgn}}(n)}}})_{pb}$ with mixed boundary values.
Since there is no sign change, the word $s_{j_1}w_{j_1+1}^{{\rm sgn}(n)}$ is of form $(2)$.

{\bf  If the left boundary values of $w$ are $tr$} then $g\mid R_1$ represents $\mathfrak{s}_1$ with mixed boundary values.
Consider the most left syllable $\mathfrak{s}_{j_1}$ with $1<j_1<k$ which is not of form (1)   (if there is any) and proceed as in the previous case .

{\bf The inductive procedure.} 
We obtained the following facts. If there is a syllable  $\mathfrak{s}_{j_1}$ not of form (1) on the left of
$\mathfrak{s}_k$ then Statement 2 holds for all syllables not of form (1) with label
$j\leq j_1+1$. The disjoint curvilinear rectangles contained in $R$
used for representing these syllables (or syllables with one letter
more) are contained in the union of the closure of the rectangles
$R_j,1\leq j\leq j_1+1,$ of Statement 1 of the lemma. Notice that the restriction $g\mid {\rm Int}(\bigcup_{j=1}^{j_1+1} \overline{R_j}) $ has $pb$ right boundary values.

We proceed by induction as follows. Suppose for some
$l<k$ we achieved the following. We found disjoint curvilinear
rectangles contained in $\bigcup_{j\leq l} \overline{R}_{j} $
which are
in bijective correspondence to all syllables
$\mathfrak{s}_j$ with $j\leq l$ that are not of form (1) such that
the restrictions of $g$ to these rectangles represent the syllable, or
a syllable with one more letter, with mixed boundary values. Moreover,
the restriction $g\mid {\rm Int}(\bigcup_{j=1}^{l} \overline{R_j}) $ has $pb$ right boundary values.

Consider the first syllable $\mathfrak{s}_{j_2}$ with $j_2<k$ on the right
of $\mathfrak{s}_{l}$ with $\mathfrak{s}_{j_2}$
not of form (1) (if there is any). If $j_2+1<k$ we
proceed with $\mathfrak{s}_{j_2}$ in the same way as it was done
for $\mathfrak{s}_{j_1}$ in the case $j_1+1< k$ and continue the process.
If there is no such syllable we stop the process.

Make the same procedure from right to left until each
syllable not of form (1) on the right of $\mathfrak{s}_k$ is
represented in the
desired way. The curvilinear rectangles obtained by the construction
which starts from the left do not intersect the curvilinear rectangles
obtained by the construction which starts from the right because
$\mathfrak{s}_k=w_k^{n_k}$ with $|n_k|\geq 2$.
Statement 2 is proved.

\noindent {\bf Proof of Statement 3.} Under the conditions of Statement 3 there is a sign change of exponents for any pair of
consecutive syllables.

If the left boundary values of $w$ are $pb$ we
may consider curvilinear rectangles $R_{2j-1,2j}, \, j=1,2,\ldots ,$
such that
the restriction of $g$ to $R_{2j-1,2j}, \, j=1,2,\ldots , $ represents
$_{pb}(\mathfrak{s}_{2j-1} \,\mathfrak{s}_{2j})_{pb}$. By Lemma \ref{lemm15'} Statement 2 there exists a smooth arc
that divides $R_{2j-1,2j}$ into two curvilinear rectangles
such that the restriction of $g$ to the first curvilinear rectangle $R_{2j-1}'$
represents the
syllable $\mathfrak{s}_{2j-1}$  with right $tr$  boundary values and
the restriction of $g$ to the
second curvilinear rectangle $R_{2j}'$  represents the syllable with
left $tr$ boundary values. In this way each syllable except, maybe, the last one
is represented
with mixed boundary values by restricting $g$ to a member of the
collection of the obtained  pairwise
disjoint curvilinear rectangles.

If the left boundary values of $w$ are $tr$ then $g\mid R_1$ represents $\mathfrak{s}_1$ with mixed boundary values and we consider instead the rectangles $R_{2j,2j+1},\,j=1,\ldots$. We obtained the collection $R'_j$ for both cases of the left boundary values.

Repeating the procedure from right to
left gives the rectangles $R_j''$.
This finishes the proof of Statement 3. \hfill $\Box$

\medskip
\noindent {\bf Proof of Lemma \ref{lemm15'}, Statement 1.}  We may assume that the annulus in the Statement 1 equals $A=\{z\in\mathbb{C}: 2<|z|<r\}$ for a positive number $r>2$.
Take a smooth closed curve $\beta':\partial \mathbb{D}\to \mathbb{C}\setminus \{-1,1\}$ that represents $\hat w$, intersects $i\mathbb{R}$ transversally, and has minimal possible number of intersection points with the imaginary axis. Then the intersection points of $\beta'$ with the imaginary axis divide the curve $\beta'$ into curves $_{pb}(\beta'_l)_{pb}$ whose interior does not intersect $i\mathbb{R}$.
No $_{pb}(\beta'_l)_{pb}$ is homotopic with endpoints in $i\mathbb{R}$ to a constant curve, since otherwise the curve $\beta$ would be homotopic to a curve with smaller number of intersection points with $i\mathbb{R}$. Then the   $_{pb}(\beta'_l)_{pb}$ represent words of the form $_{pb}(a_{j_l})^k_{pb}$, and
neighbouring  $_{pb}(\beta'_{l})_{pb}$ on the closed curve $\beta'$ are powers of different generators $a_{j_l}^k$ among the $a_1$ and $a_2$. Since the word $w$ is cyclically reduced, we may suppose (after perhaps relabeling the $\beta'_l$) that $w=\beta'_1 \, \beta'_ 2\,\ldots\, \beta'_N \,$.

The proof
follows now along the same lines as the proof of Statement 1 of Lemma \ref{lemm15}.  We consider a smooth homotopy that joins the curves $\beta'$ and $g\mid \{|z|=2\}$, more precisely, we consider a smooth mapping $h:\{1\leq |z|\leq r\}\to \mathbb{C}\setminus \{-1,1\}$ for which zero
is a regular value of the real part, whose restriction to $\{|z|=1\}$ coincides with $\beta'$, and whose restriction to $\bar A$ equals $g$. Similarly as in the proof of  Lemma \ref{lemm15} the division of $\beta'$ into $_{pb}(\beta'_l)_{pb}$ induces a division of the annulus $A$ into curvilinear rectangles $R'_l$ by smooth arcs in ${A}$, each with endpoints on different boundary circles of $A$ and contained in the preimage $g^{-1}(i\mathbb{R})$. The restriction $g\mid R'_l$
represents $_{pb}(\beta'_l)_{pb}$.
The arc that is the common part of the boundaries of $R'_1$ and $R'_N$ satisfies the requirement of Statement 1 of the lemma.\\

\noindent {\bf Proof of Statement 2.} Since there is a sign change of exponents for the two terms of the word, we may represent  $_{\#}({a_{j_1}^{n_1}\,a_{j_2}^{n_2}})_{\#}$ by a smooth curve that intersects the real axis transversally and equals the product $_{\#}(\beta'_1)_{tr}\, _{tr}(\beta'_2)_{tr}       \ldots_{tr}(\beta'_{n_1+n_2})_{\#}$ of curves, where each of the first $j_1$ curves represents  $a_{j_1}^{{\rm sgn}(n_1)}$ and the remaining curves represent  $a_{j_2}^{{\rm sgn} (n_2)}$, and the interior of each curve
does not intersect $\mathbb{R}$. The proof follows now along the lines of proof of Statement 1 of Lemma \ref{lemm15}.\\

\noindent {\bf Proof of Statement 3.}
We choose a representative $\beta'$ of $_{\#}(a_j^n)_{\#}$
that has minimal number of intersection points with the real axis.
The curve $\beta'$ can be written as $_{\#}(\beta'_1)_ {tr}\ldots _{tr}(\beta'_{n-1})_{\#}$ for curves $\beta_j'$ that represent $_{\#}(a_j)_{\#}$ whose interior does not intersect $(-1,1)$.
The proof now follows along the same lines of arguments as in the proof of Statement 1 of Lemma \ref{lemm15}.  \\
The lemma is proved. \hfill $\Box$

\medskip

We will  now give the proof of the lower bound in Theorem \ref{thm1}
and in Theorem \ref{thm10.1'}. The plan is the following. We consider a locally conformal mapping in a neighbourhood of a closed rectangle $R$, that represents an element $w\in \pi_1^{pb}$. The rectangle will be covered by the closures of curvilinear subrectangles for which the following holds. Each point of $R$ is contained in at most two subrectangles. The restriction of the mapping to each subrectangle $R_j$ represents a syllable $\mathfrak{s}_j$ for which the horizontal boundary values are chosen so that Proposition \ref{prop4b} provides
a suitable lower bound of the extremal length of the subrectangle in terms of $\mathcal{L}(\mathfrak{s}_j)$ .

\noindent Figure  \ref{Fig3-braids5.pdf}  below illustrates a mapping $g$ from a rectangle $R$ to $\mathbb{C} \setminus i\mathbb{Z}$ which represents a lift under $f_1 \circ f_2$ of an element $w \in \pi_1$ with $pb$ boundary values. The rectangle $R$ is covered by relatively closed curvilinear rectangles. For each curvilinear rectangle we indicate the word with suitable horizontal boundary conditions (the ''building block'') whose lift under $f_1 \circ f_2$ is represented by
the restriction of $g$ to the curvilinear rectangle. Notice that for some choices of the boundary values the curvilinear rectangles may intersect, but each point of $R$ is contained in no more than two relatively closed curvilinear rectangles.

The right part of the figure shows the image $g(R)$ of $R$ in $\mathbb{C} \setminus i\mathbb{Z}$
and a slalom curve that lifts a representative of $w$. For the elementary pieces of the slalom curve we indicate the element of $\pi_1$ a lift of which the piece represents.

\begin{figure}[H]
\begin{center}
\includegraphics[width=12cm]{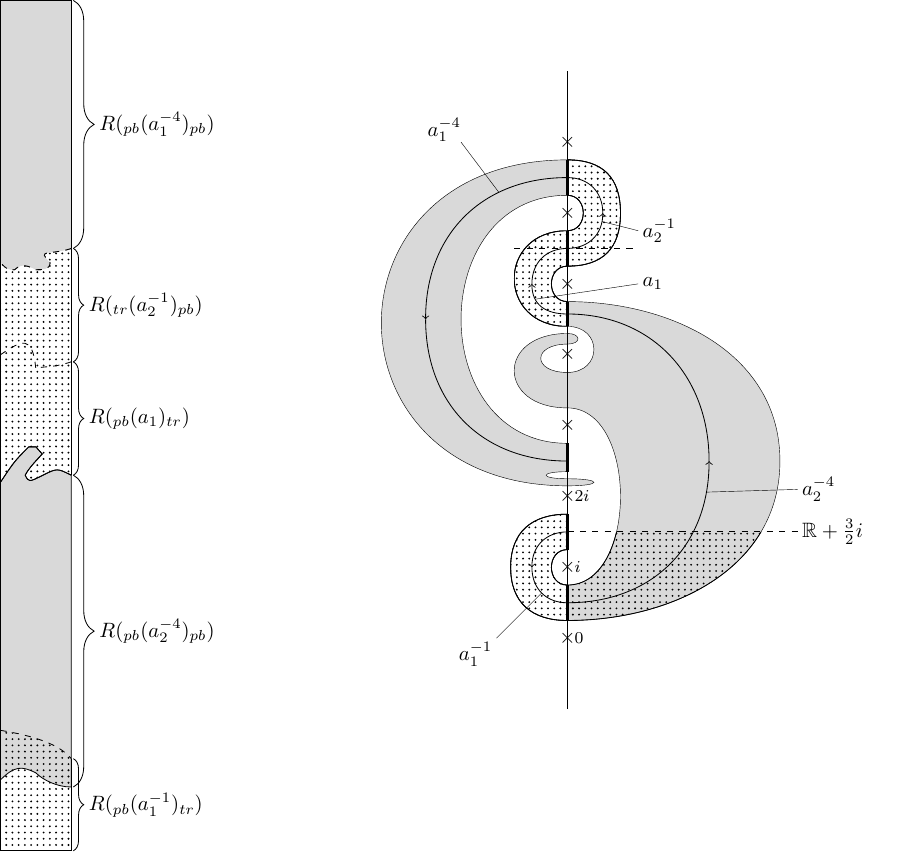}
\end{center}
\caption{A representative of a lift to $\mathbb{C}\setminus i\mathbb{Z}$  of an element of $\pi_1$ and building blocks}\label{Fig3-braids5.pdf}
\end{figure}

\noindent {\bf Proof of the lower bound in Theorem \ref{thm1} and in Theorem \ref{thm10.1'}.}
By Lemma \ref{cor2} the extremal length of an element of a relative fundamental group of $\mathbb{C}\setminus \{-1,1\}$ is attained on
a locally conformal mapping $g$ from an open rectangle $R$ with sides parallel
to the axes into $\mathbb{C} \setminus \{-1,1\}$, which  represents the
word $w$ in the relative fundamental group of the twice punctured
complex plane with $tr$, $pb$, or mixed horizontal boundary values.
Since the continuous extension of
$g$ maps each horizontal side to the real axis or to the imaginary axis,
we may extend $g$ holomorphically through the horizontal sides. Shrinking the rectangle slightly in the horizontal direction (and hence, increasing the extremal length slightly), we may assume, that the mapping $g$ is locally conformal in a neighbourhood of the closed rectangle $\bar R$.

We assume first that the word has at least two syllables,
and the word has at least one syllable of form (1). Let
$I$ be the set of the natural numbers $j$ for which the $j$-th
term of $w$ is a syllable of form (1) and let $R_j, \, j \in I,$ be the curvilinear
rectangles of Statement (1) of Lemma \ref{lemm15}.

By Corollary \ref{cor4.0}
\begin{align*}
 \sum _{j \in I} \lambda(R_j) \leq \lambda(R)\,.
\end{align*}
The rectangle ${R}_j$ admits the
holomorphic mapping $g \mid R_j$
to
$\mathbb{C} \setminus \{-1,1\}$ which represents the syllable
corresponding to $R_j$ with $pb$   boundary values if the syllable is not at the end of the word, and, possibly, with mixed boundary values if the syllable is at the end of the word. Let $d_j$ be the degree of the $j$-th syllable.
By Proposition  \ref{prop4b}  the extremal length
of  ${R}_j$ is not smaller than
$\frac{1}{\pi}\log(2d_j+\sqrt{4d_j^2-1})$
if $g\mid R_j$ has mixed boundary values and $d_j\geq2$, and it is not smaller
than
\begin{align}\label{eq3braids.101}
\frac{2}{\pi}\log(d_j+\sqrt{d_j^2-1})\geq \frac{1}{\pi}\log(2d_j+\sqrt{4d_j^2-1})
\end{align}
if $g\mid R_j$ has $pb$ boundary
values and $d_j\geq2$.

We obtain
\begin{equation}\label{eq8}
\frac{1}{2}\lambda(R)\geq \frac{1}{2}\sum _{j \in I}
\lambda(R_j) \geq
\frac{1}{2}\sum _{j \in I} \frac{1}{\pi}\log(2d_j+\sqrt{4d_j^2-1})\,.
\end{equation}
Suppose as before that $w$ contains at least one syllable of form (1).
Denote by $J$ the set of all natural numbers $j$ for which
$\mathfrak{s}_j$ is
not of form (1). For $j \in J$ we denote
by $ R_j'$ the curvilinear rectangle of Statement 2 of Lemma \ref{lemm15} corresponding to $\mathfrak{s}_j$
(or to an elementary word of form $(2)$ with one more letter than $\mathfrak{s}_j$).

Since $g\mid R_j'$ has mixed boundary values,
Proposition  \ref{prop4b}  implies that the extremal length
of  $R_j'$ is not smaller than $\frac{1}{\pi}\log(2d_j+\sqrt{4d_j^2-1})$.
(If $R_j'$ represents an elementary word of form $(2)$ with one letter more than $\mathfrak{s}_j$ then the lower bound is even bigger, namely, $d_j$ may be replaced by $d_j+1$.)
As before $d_j\geq 1$ is the degree of
the syllable $\mathfrak{s}_j$.
Hence,
\begin{equation}\label{eq9}
\frac{1}{2}\lambda(R)\geq \frac{1}{2}\sum _{j \in J}
\lambda(R_j') \geq
\sum_{j \in J} \frac{1}{2\pi}\log(2d_j+\sqrt{4d_j^2-1})\,.
\end{equation}
For the case when  $w$ contains syllables of form (1) and syllables not of form (1) we add the two inequalities
\eqref{eq8} and \eqref{eq9}. We obtain
\begin{align}\label{eq10}
\lambda(R) \geq & \sum_{j \in J} \frac{1}{2\pi} \log(2d_j+\sqrt{4d_j^2-1})
+ \sum_{j \in I} \frac{1}{2\pi} \log(2d_j+\sqrt{4d_j^2-1})\nonumber \\
= &\sum _{\mathfrak{s}_j}\frac{1}{2\pi}\log(2d_j+\sqrt{4d_j^2-1})\,.
\end{align}
The last sum is extended over all syllables  $\mathfrak{s}_j$ of the word $w$, and $d_j$ is the degree of $\mathfrak{s}_j$.

Suppose that the word does not contain syllables of form (1) and the
syllables
are labeled from left to right by $j=1,\ldots,N$. Let $R_j'$ be the
rectangles from Statement 3 of Lemma \ref{lemm15}.
Since $g\mid R'_j$ represents $\mathfrak{s}_j$with mixed boundary values we obtain
\begin{equation}\label{eq11}
\lambda(R) \geq \sum_{j=1}^{N-1} \lambda(R_j') \geq \sum
_{j=1}^{N-1} \frac{1}{\pi}\log(2d_j+\sqrt{4d_j^2-1})\,.
\end{equation}

On the other hand, with the curvilinear rectangles $R_j''$ from
Statement 3 of Lemma \ref{lemm15}
we obtain the
inequality
\begin{equation}\label{eq12}
\lambda(R) \geq \sum_{j=2}^{N} \lambda(R_j'') \geq \sum
_{j=2}^{N}\frac{1}{\pi} \log(2d_j+\sqrt{4d_j^2-1})\,.
\end{equation}
It follows that for any word $w \in \pi_1$ with at least two syllables for the respective family $\Gamma$ the
inequality

\begin{equation}\label{eq13}
\lambda(R) \geq   \sum _{\mathfrak{s}_j}
 \frac{1}{2\pi} \log(2d_j+\sqrt{4d_j^2-1})\,
\end{equation}
holds.

If the word consists of a single syllable, Proposition \ref{prop4b} implies \eqref{eq13} in the case of mixed boundary values as well as in the non-exceptional cases with both boundary values being $tr$ or $pb$.

Consider the exceptional cases. If $w=a_1^n$ with $n>0$ the mapping $\zeta \to -1+ e^{\zeta}, \; \zeta \in R,$ with $R=\{ \xi + i \eta,\, \xi \in (-\infty,0),\, \eta \in (0, 2\pi n)\}$, represents $w_{tr}$. Hence,  $\Lambda(w_{tr})=0$.

If $w=a_1 a_2\ldots$ and has degree $d \geq 2$ then the mapping $\zeta \to e^{\zeta}, \; \zeta \in R,$ with $R=\{ \xi + i \eta,\, \xi \in (1,\infty),\, \eta \in (\frac{\pi}{2},\frac{\pi}{2} + \pi d)\}$, represents $w_{pb}$. Hence, $\Lambda(w_{pb})=0$.

The other exceptional cases are similar.
The lower bound in Theorem \ref{thm1} and in Theorem \ref{thm10.1'} is proved. \hfill $\Box$

\medskip
It will be useful to have in mind the following corollary.
\begin{cor}\label{lem3-braids1} Let $\hat w$ be a conjugacy class
of elements of $\pi_1$ that is not among the exceptional cases of Theorem \ref{thm2}.\\
Then for any cyclically reduced representative $w \in  \pi_1$ of $\hat w$ the inequality
\begin{align}\label{eq14}
\Lambda(w_{pb}) \leq \Lambda(\hat w)
\end{align}
holds.  The inequality may be strict.

Further, let $w$  be any representative of $\hat w$, that is obtained by a cyclic permutation of a cyclically reduced representative of $\hat w$ and has one of the following properties. Either the first and the last term of $w$ are powers of the same sign of the same standard generator of $\pi$, or  the first and the last term of $w$ are powers of different sign of different standard generators of $\pi_1$. Then
\begin{align}\label{eq14'}
\Lambda(w_{tr}) \leq \Lambda(\hat w)\,.
\end{align}
The inequality may be strict.
\end{cor}

\noindent {\bf Proof.} Let $\hat w\neq \widehat{a_j^k}$
and $\hat g$  a locally conformal mapping from an annulus $A$ into the twice punctured complex plane $\mathbb{C}\setminus \{-1,1\}$ that represents $\hat w$.
We may assume that $\hat g$ is  locally conformal in a neighbourhood of $\bar A$
and $\lambda(A)<\Lambda(\hat{w})-\varepsilon$ for an a priori given small positive number $\varepsilon$ .
By Statement 1 of Lemma \ref{lemm15'} there is a smooth arc $L_0\subset \bar A$
with endpoints on different boundary circles of $A$, such that $\hat{g}\mid A\setminus L^0$ represents  the cyclically reduced word $w$ with $pb$ boundary values.
Hence, by Corollary  \ref{cor4.0}
the inequality
$\Lambda(w_{pb})\leq\lambda(A\setminus L^0)\leq \lambda(A)\leq\Lambda(\hat{w})  -\varepsilon$. This proves inequality \eqref{eq14}.

Inequality \eqref{eq14'} is proved in the same way using Corollary  \ref{cor4.0} and
the analog of Statement 2  or Statement 3  of Lemma \ref{lemm15'} for annuli instead of rectangles.
Examples 2 and 3 of Section 4.2 show that the inequality may be strict. \hfill $\Box$

\medskip

\noindent {\bf Proof of the lower bound of $\Lambda(\hat{w})$ in Theorem \ref{thm2}.}
The lower bound of $\Lambda(\hat{w})$ in Theorem \ref{thm2} for the non-exceptional case follows immediately from Lemma \ref{lemm15'}, Statement 1,
Theorem  \ref{thm1}, and Corollary  \ref{cor4.0}.

Similarly as in the proof of the lower bound in Theorem  \ref{thm1} we see that
in the exceptional cases the extremal length of $\hat b$ equals zero. \hfill $\Box$

\section[Upper bound of extremal length. All syllables of degree $2$.] {The upper bound of the extremal length for words whose syllables have degree $2$.}\label{sec:3-braids3}
We first obtain the upper bound of the extremal length of reduced words of the form
\begin{align}\label{eq3-braids1}
 w= a_1^{\pm 2}\,a_2^{\pm 2}\, \ldots
\end{align}
with at least two syllables and $pb$, $tr$, or mixed horizontal boundary values, because the proof in this case avoids technical details and gives a better estimate than in the general case. We will use this estimate
in Chapter \ref{Ch9}. The proof in the general situation is given in Section \ref{sec:3-braids4}.

We first represent the lift under $f_1\circ f_2$ of the elementary word
$_{pb}(a_2^{- 2})_{pb}$  by a holomorphic mapping $g_{a_2^{- 2}}$ of a rectangle $R_{a_2^{- 2}}$ into $\mathbb{C}\setminus i\mathbb{Z}$.

Consider the rectangle $R_{a_2^{- 2}}=\sf R$
in the plane with vertices $\pm  \frac{1}{5} \pm i \frac{\pi}{2} $ and the mapping
\begin{equation}\label{eq302'}
g_{a_2^{-2}}(\zeta)=\frac{i}{2}+ e^{\zeta},\;\zeta \in {\sf R}\, .
\end{equation}
The mapping $g_{a_2^{-2}}$ takes the rectangle $\sf R$
conformally onto a half-annulus in the right half-plane.
It maps the point $-\frac{\pi i}{2}$  to $-\frac{i}{2}$, and  maps $\frac{\pi i}{2}$ to  $\frac{3 i}{2}$.
The upper side of the rectangle is mapped to the interval $\frac{i}{2}+i( e^{-\frac{1}{5}},e^{\frac{1}{5}})$. This interval is contained in $(i, 2i  )$, and has distance to the endpoints of this interval equal to $\min \{e^{-\frac{1}{5}}-\frac{1}{2}, \frac{3}{2}-
e^{\frac{1}{5}}\}\geq 0.25$.
The lower side is mapped to the interval $\frac{i}{2}+i(- e^{\frac{1}{5}},-e^{-\frac{1}{5}})$, that is contained in $(-i,0)$ and has distance from the endpoints of $(-i,0)$ at least equal to
$0.25$.
Hence, the mapping represents a lift of $_{pb}(a_2^{- 2})_{pb}$ under $f_1\circ f_2$. Its image has distance at least $\frac{1}{4}$ to $i \mathbb{Z}$.
Moreover, $g_{a_2^{-2}}'(-\frac{\pi i}{2})= -i$,  $g_{a_2^{-2}}'(\frac{\pi i}{2})= i$, ${\rm vsl}({\sf R})=\pi$, ${\rm hsl}({\sf R})=\frac{2}{5}$.
The mapping  $g_{a_2^{-2}}$ extends holomorphically across the horizontal sides of ${\sf R}$. The extension is denoted again by  $g_{a_2^{-2}}$.

A lift of the elementary word $_{pb}(a_2^{2})_{pb}$ under $f_1\circ f_2$  is represented by the mapping $g_{a_2^{2}}(\zeta)\stackrel{def}=   g_{a_2^{-2}}(-\zeta)\,,$
$\zeta\in \sf R$, whose derivatives at $\pm \frac{\pi i}{2}$ coincide with those of $ g_{a_2^{-2}}$ at these points, and lifts of the words $_{pb}(a_1^{\pm 2})_{pb}$ under $f_1\circ f_2$ are represented by the mappings $g_{a_1^{\pm 2}}(\zeta)\stackrel{def}=-  g_{a_2^2}(\pm\zeta),\,$ $ \zeta\in \sf R$, whose derivatives at $\pm \frac{\pi i}{2}$ are $\mp i$.

A lift under $f_1\circ f_2$  of the syllable $_{tr}(a_2^{-2})_{pb}$  with mixed horizontal boundary values can be represented by the mapping $g_{_{tr}(a_2^{-2})_{pb}}(\zeta)= -\frac{i}{2}+ 2e^{\frac{\zeta}{2}}$ on the rectangle $R_{_{tr}(a_2^{-2})_{pb}}={\sf R}_{\rm mix}\stackrel{def}=\{\zeta \in \mathbb{C}: {\rm Im}\zeta \in(0,\pi),\, |{\rm Re}\zeta|<\frac{1}{5}\}$. Indeed,  $g_{_{tr}(a_2^{-2})_{pb}}$ maps
the rectangle  $R_{_{tr}(a_2^{-2})_{pb}}$ into the right half-plane, $0$ is mapped into $-\frac{i}{2} +\mathbb{R}$, and $\pi i$ is mapped to $\frac{3i}{2}$. Moreover, the lower side of the rectangle is mapped to a relatively compact subset of  $-\frac{i}{2} +\mathbb{R}$, and the upper side is mapped to
the interval $-\frac{i}{2}+2i( e^{-\frac{1}{10}},e^{\frac{1}{10}})$. The interval is contained in $(i,2i)$ and has distance to the endpoints of $(i,2i)$ equal to $\min\{-\frac{3}{2}+2 e^{-\frac{1}{10}},\frac{5}{2}-2 e^{\frac{1}{10}}\}\geq 0.3$.

For the derivative of the mapping the equality
 $g_{_{tr}(a_2^{-2})_{pb}}'(i\pi)=i $ holds. Moreover,  ${\rm vsl}(R_{_{tr}(a_2^{-2})_{pb}})=\pi$ and ${\rm hsl}(R_{_{tr}(a_2^{-2})_{pb}})=  \frac{2}{5}$.

Lifts under $f_1\circ f_2$ of all syllables of the form \eqref{eq3-braids1} and degree $2$  with mixed horizontal boundary values can be represented in a similar way by holomorphic mappings on
${\sf R}_{\rm mix }$.

There exists a universal constant $C'>1$ such the following holds.

For each $j=1,2$, each choice of the sign, and each integral number $k$ the compositions  $\;f_1\circ f_2\circ (g_{a_j^{\pm 2}}+ik)\mid {\sf R}\;,\;$
$\;f_1\circ f_2\circ (_{tr}(g_{a_j^{\pm 2}})_{pb}+ik)\mid {\sf R}_{\rm mix }\;,\;$ and  $\;f_1\circ f_2\circ (_{pb}(g_{a_j^{\pm 2}})_{tr}+ik)\mid {\sf R}_{\rm mix }\;\;$
have its image in the domain $\{z\in\mathbb{C}: |z|<C', |z\pm 1|>\frac{1}{C'}\}$. Indeed, $g_{a_j^{\pm 2}}( {\sf R})$,
$_{pb}(g_{a_j^{\pm 2}})_{tr}({\sf R}_{\rm mix })$ and
$_{tr}(g_{a_j^{\pm 2}})_{pb}({\sf R}_{\rm mix })$
have distance at least $\frac{1}{4}$ from
$\mathbb{C}\setminus i\mathbb{Z}$, hence the statement is true for $k=0$. For arbitrary integer numbers $k$ it follows from the periodicity of $f_1\circ f_2$.

Consider the word $_{pb}w_{pb}=_{pb}(a_1^{\pm 2} a_2^{\pm 2}\dots)_{pb}$ with $N\geq 2$ syllables. Denote the $j$-th syllable by $\mathfrak{s}_j$. Represent the lift under $f_1\circ f_2$ of each   $_{pb}(\mathfrak{s}_j)_{pb}$ by the mapping $g_{\mathfrak{s}_j}:{\sf R}\to \mathbb{C}\setminus i\mathbb{Z}$ described above.

Put $R_j\stackrel{def}={\sf R}+\pi i c_j$ for integer numbers $c_j$, so that for the midpoints $\xi_j^-$ and $\xi_j^+$, respectively, of the lower and upper side of $R_j$ the equalities $\xi_j^+={\xi}_{j+1}^- , \, j=1,\ldots N-1,$ hold, where $N$ is the number of syllables of the word $w$. Take $g_j(\zeta)\stackrel{def}=g_{\mathfrak{s}_j}(\zeta-\pi i c_j)+\pi i b_j,\, \zeta \in \overline{R_j},$
for integer numbers $b_j$ such that $g_j(\xi_j^+)=g_{j+1}(\xi_{j+1}^-)$ .

The rectangle $R_w\stackrel{def}={\rm Int}(\cup\overline{R_j})$
has vertical side length ${\rm vsl}(R_w)=\pi N$ and horizontal side length $\frac{2}{5}$.
The mappings $g_j:R_j\to\mathbb{C}\setminus i\mathbb{Z}$ represent  $_{pb}(\mathfrak{s}_j)_{pb}$ and extend by the Reflection Principle holomorphically across the horizontal sides
to the rectangle of thrice the vertical side length, with the same horizontal side length and the same center as $R_j$.
The extended mappings are denoted again by $g_j$. The distance of the images of the extended mappings from $i\mathbb{Z}$ is the same as that of the original mappings.
We will perform quasiconformal gluing of the $g_j$ using the fact that $g_j( \xi_j^+)=g_{j+1}(\xi_{j+1}^-)$ and $g'_j( \xi_j^+)=g'_{j+1}(\xi_{j+1}^-)$.
The latter statement follows from the fact that one of two consecutive syllables is a power of $a_1$ and the other is a power of $a_2$.

Consider the $C^1$-function $\chi_0$ on the interval $[0,1]$, $\chi_0(t)=6 \int_0^t \tau (1-\tau) d\tau$. Then $\chi_0(0)=0$, $\chi_0(1)=1$, and $0\leq \chi_0'(t)= 6t(1-t)\leq \frac{3}{2}$.
Put $\chi(t)=\chi_0(\frac{5}{2} t)$.
Define a function $g_w$ on $R_w$ as follows.
Each point $\xi$ in $R_w$, for which $|\mbox{Im}\,(\xi -\xi_j^+)|> \frac{1}{5}$ for all $j<k$,
belongs to a single rectangle $\overline{R_j }$ (depending on $\xi$) and we put $g_w(\xi)\stackrel{def}{=}g_j(\xi)$ for such a point.
\index{$\chi_0$}

Fix a number $j<N$
and consider the set $Q_j\stackrel{def}{=}\{\xi\in R_w:|\mbox{Im}\,(\xi -\xi_j^+)|\leq \frac{1}{5}\}$. Put
$\chi_j(\xi)=\chi(\mbox{Im}\,(\xi -\xi_j^+)+\frac{1}{5})\, $ for $\xi \in  Q_j$. Let $g_w\stackrel{def}{=} (1-\chi_j)  g_j + \chi_j  g_{j+1}$ on $ Q_j$.
For $\{\xi \in Q_j: \mbox{Im}\,\xi = \xi_j^+ -\frac{1}{5}\}$ the equalities
$\chi_j(\xi )= \chi_0(0)=0$ and $\chi'_j(\xi)=\chi_0'(0)=0$ hold. Hence, the function $g_w$ is $C^1$ smooth near such points $\xi$. Further, for $\{\xi \in Q_j:\mbox{Im}\,\xi = \xi_j^+ +\frac{1}{5}\}$ the equalities
$\chi_j(\xi)= \chi_0(1)=1$ and $\chi_j'(\xi)=\chi_0'(1)=0$ hold, hence, the function $g_w$ is $C^1$ smooth near such $\xi$.

Since $g_{j+1} - g_j$ and $g'_{j+1} - g'_j$ vanish at $\xi_j^+=\xi_{j+1}^-$, and the absolute value of the second derivative of $g_j$ and of $g_{j+1}$ on $Q_j$ does not exceed
$e^{\frac{1}{5}}<1.222$, the estimate
\begin{align}\label{eq3braids25}
|(g_{j+1} - g_j)(\xi)|\leq 2 e^{\frac{1}{5}}\frac{|\xi-\xi_j^+|^2}{2} \leq  e^{\frac{1}{5}} \frac{2}{25} <\frac{1}{10}
\end{align}
holds on $Q_j$.

Make the same definition for all but the last number $j$. We obtain a smooth mapping $g_w$ from the rectangle $\overline{ R_w}$ to $\mathbb{C} \setminus i\mathbb{Z}$ which represents a lift of $_{pb}w_{pb}$ under $f_1 \circ f_2$.
Since the distance of
the image of each mapping $g_j$ from $i\mathbb{Z}$ is not smaller than $\frac{1}{4}$, by inequality \eqref{eq3braids25} the distance of the image  $g_w(Q_j)$  from $i\mathbb{Z}$ is not smaller than $0.15$.
It follows that the image $f_1\circ f_2\circ g_w(\overline{ R_w})   $ is contained in the closure of a domain $\{z\in\mathbb{C}: |z|<C, |z\pm 1|>\frac{1}{C}\}$ with another universal constant $C$.

\begin{lemm}\label{lemm221} The mapping $g_w$ is a quasiconformal mapping from $R_w$ onto its image. The Beltramy differential $\mu_{g_w}$ of $g_w$ has absolute value $|\mu_{g_w}|< \frac{2}{5}$.
\end{lemm}

\medskip
\noindent {\bf Proof}. Put $\xi=u+i v$. If $\xi \in R_w$, $|\mbox{Im}\, (\xi -\xi_j^+)|\leq \frac{1}{5}$ for some $j, \, 1 \leq j <N,\,$  then the Beltrami differential at $\xi$ equals
\begin{align}\label{eq3braids25A}
\mu_{g_w}(\zeta)= \frac{\frac{\partial}{\partial \overline \zeta}g_w(\zeta)}{\frac{\partial}{\partial \zeta}g_w(\zeta)} = \frac{\frac{i}{2} \Big(\frac{\partial}{\partial v}{\chi_j}\cdot (g_{j+1} -g_{j})\Big)(\zeta)}{\Big(\frac{-i}{2} \frac{\partial}{\partial v}{\chi_j}\cdot (g_{j+1} -g_{j})+ (1-\chi_j)\cdot g_j' + \chi_j\cdot g_{j+1}'\Big)(\zeta)}.
\end{align}
On the rest of the rectangle $R_w$ the function is analytic.
Recall that
$$
|\frac{\partial}{\partial v}{\chi_j}|\leq\frac{3}{2} \cdot \frac{5}{2}
$$
on $Q_{j}$. By inequality \eqref{eq3braids25} the numerator in the right hand side of
 \eqref{eq3braids25A} has absolute value smaller than
${e^{\frac{1}{5}} \cdot \frac{3}{20}      }<0.1833$.

Since $\max|g''_j|\leq e^{\frac{1}{5}} $ on $Q_{j}$, the inequality
$$
\max\{|g_j'-g_j'(\xi_j^+)|,|g_{j+1}'-g_j'(\xi_j^+)|\}< e^{\frac{1}{5}} \cdot \frac{\sqrt{2}}{5}
$$
holds on the  $\frac{\sqrt{2}}{5}$-neighbourhood of $\xi_j^+$.
Since $g_j'(\xi_j^+)=g_{j+1}'(\xi_{j+1}^-)$ has absolute value $1$, the denominator of the right hand side of \eqref{eq3braids25} is not smaller than
$1- e^{\frac{1}{5}} \cdot \frac{3}{20}        -e^{\frac{1}{5}} \cdot \frac{\sqrt{2}}{5}  > 1- 0.53= 0.47$.

We obtain
$$
k_w = \sup_{R_w} |\mu_{g_w}(\zeta)| < \frac{0.1833}{0.47}<0.4\,.$$
The quasiconformal dilatation $K_w=\frac{1+k_w}{1-k_w}$ does not exceed $\frac{7}{3}$.
\hfill $\Box$

\medskip

Let $\omega_w$ be the normalized solution of the Beltrami equation
$$
\frac{\partial}{\partial \overline z} \omega_w = \tilde \mu_{g_w} \frac{\partial}{\partial z} \omega_w
$$
on the complex plane. Here $\tilde \mu_{g_w}$ equals $\mu_{g_w}$ on $\overline{ R_w}$ and equals $0$ outside $ \overline{ R_w}$. $\omega_w$ is a H\"older continuous self-homeomorphism of the complex plane. The mapping $g_w \circ  {\omega_w}^{-1}$ is holomorphic on $\omega_w(R_w)$ (see \cite{A1}, Chapter I C).
The image $\omega_w(R_w)$  can be considered as a curvilinear rectangle. The curvilinear sides are the images of the sides of $R_w$. By \cite{A1} (chapter I, Theorem 3) the extremal length of $\omega_w(R_w)$  does not exceed $K_w \cdot \lambda(R_w)$. In other words, there is a conformal mapping $\psi_w$ of a true rectangle $\mathcal{R}_w$ of extremal length not exceeding  $K_w \cdot \lambda(R_w)$ onto $\omega_w(R_w)$, which takes the sides of $\mathcal{R}_w$ to the respective curvilinear sides of $\omega_w(R_w)$ . The mapping $g_w\circ {\omega_w}^{-1} \circ \psi_w: \mathcal{R}_w \to \mathbb{C} \setminus i \mathbb{Z}$ is a holomorphic mapping from the rectangle $\mathcal{R}_w$ of extremal length not exceeding  $K _w\cdot \lambda(R_w)$ to $\mathbb{C}\setminus i\mathbb{Z}$ that represents a lift of $_{pb}(w)_{pb}$.
Notice that the image of $f_1\circ f_2 \circ g_w\circ {\omega_w}^{-1} \circ \psi_w$ is contained in $\{z\in\mathbb{C}: |z|<C, |z\pm 1|>\frac{1}{C}\}$.

We obtained the inequality
\begin{align*}
\lambda(\mathcal{R}_w)\leq K_w \frac{ \pi N }{\frac{2}{5}} \leq
\frac{7}{3}\, \cdot \frac{ 5}{ 2}  \pi N =\frac{35}{6}  \pi N
\end{align*}
Since $\mathcal{L}(w)=N \cdot \log(4+\sqrt{15})$ and $\frac{35 \pi}{6\log(4+\sqrt{15})}< 9$
we obtain
$$
\Lambda_{pb}(w) <9 \mathcal{L}(w)\,.
$$
The case of words of the form \eqref{eq3-braids1}  with $tr$ or mixed boundary values is treated in the same way.
For instance, the first syllable of the word may have totally real left boundary values. But it has always $pb$ right boundary values.
We described a holomorphic mapping from a rectangle of vertical side length $\pi$ and horizontal side length $\frac{2}{5}$, that represents the
lift of such a syllable under $f_1\circ f_2$. This representing mapping has the same properties near the upper side of the rectangle as the chosen representing mapping of the lift of the syllable with $pb$ boundary values. Therefore the gluing procedure is the same as in the case of $pb$ boundary values. The argument concerning the last syllable is the same with left and right boundary values interchanged.
We proved the following
\begin{prop}\label{propxx}
If  $w\in \pi_1$ is a word with $N\geq 2$ syllables, all of the form $a_1^{\pm 2}$ or $a_2^{\pm 2}$, then
\begin{align*}
\Lambda_{pb} (w)\leq \frac{35}{6}\pi N< 9  \mathcal{L}(w)\,.
\end{align*}
The same estimate holds for the extremal length with $tr$ and with mixed boundary values.
\end{prop}

The following proposition is obtained in the same way. The only difference is that we have to perform also quasiconformal gluing of
the last syllable with $pb$ horizontal boundary values to the first syllable of the word with $pb$ horizontal boundary values, solve the Beltrami equation on an annulus and consider a conformal mapping from a round annulus to a curvilinear annulus.
\begin{prop}\label{propxxx}
Let $w$ be cyclically reduced word with $N\geq 2$ syllables, all of the form $a_1^{\pm 2}$ or $a_2^{\pm 2}$ (in particular, $w$ has an even number of syllables).
Then the  upper bound of the
extremal length of the free homotopy class $\hat w$ is given by the following inequality
$$
\Lambda(\hat{w})\leq \frac{35}{6}\pi N\leq 9 \mathcal{L}(\hat{w})\,.
$$
\end{prop}
The proofs of the propositions gave the following slightly more comprehensive statement that will be used later.
\begin{rem}\label{rem3-braids.1}
Let $R$ be a rectangle in the complex plane with horizontal side length $\sf b$ and vertical side length $\sf a$. Then for any word $w$ of the form  \eqref{eq3-braids1} with $N$ terms and
$6 N\pi<\frac{{\sf  a}}{\sf b}$ there exists a holomorphic map $g_w:R\to \mathbb{C}\setminus \{-1,1\}$ that represents $w$ with $pb$ horizontal boundary values and has  its image in $\{z\in \mathbb{C}: |z|<C,\,|z\pm 1|>\frac{1}{C}\}$ for a universal constant $C>1$. Moreover, $g_w$ takes the value zero at the midpoints of the horizontal sides of the rectangle.

Further, for positive numbers $\;\alpha\;$ and $\;\delta\;$ we consider the annulus
$\;A^{\alpha,\delta}\stackrel{def}= \{z\in \mathbb{C}: |{\rm Re}z|<\frac{\delta}{2}\}\diagup (z\sim z+i\alpha)\,$. If  $\lambda(A^{\alpha,\delta})=\frac{\alpha}{\delta}>6 N\pi$ for an even number $N$, then
for each word $w$ of the form  \eqref{eq3-braids1}
with $N$ terms there exists
a holomorphic mapping $\mathfrak{g}_w:  (A^{\alpha,\delta},0\diagup (z\sim z+i\alpha)) \to (\mathbb{C}\setminus \{-1,1\},0)$ from $A^{\alpha,\delta}$ with base point $0\diagup (z\sim z+i\alpha)$ to  $\mathbb{C}\setminus \{-1,1\}$ with base point $0$, that represents $\hat w$. Moreover, the image of $\mathfrak{g}_w$ is contained in the domain  $\{z\in\mathbb{C}: |z|<C, |z\pm 1|>\frac{1}{C}\}$ for a constant $C>1$.
\end{rem}
Indeed, for instance to obtain the second statement we consider the holomorphic mapping $g_w\circ {\omega_w}^{-1} \circ \psi_w$ that represents a lift of $\hat w$ and was constructed in the proofs of Propositions  \ref{propxx} and
 \ref{propxxx}. The mapping is defined on an annulus of the form $A^{\alpha,
\delta}$ with $\lambda(A^{\alpha,\delta})=\frac{\alpha}{\delta}>6N\pi$. The image $f_1\circ f_2 \circ g_w\circ {\omega_w}^{-1} \circ \psi_w(A^{\alpha,
\delta})$ under the composition  is contained in $\{z\in\mathbb{C}: |z|<C, |z\pm 1|>\frac{1}{C}\}$.
Perhaps after precomposing with a rotation of the annulus  $A^{\alpha,
\delta}$ it maps $0\diagup (z\sim z+i\alpha)$ to $0$.

\section{The extremal length of arbitrary words in $\pi_1$. The upper bound.}\label{sec:3-braids4}
We take an arbitrary word $w$ with at least two syllables and write $_{\#}w_{\#}$  as product of its syllables $_{\#}w_{\#}=_{\#}(\mathfrak{s}_1)_{pb} \, \ldots _{pb}(\mathfrak{s}_N)_{\#}$.
For each elementary word $_{pb}w_{pb}$ with $pb$ horizontal boundary conditions (or for the respective word with other horizontal boundary conditions) we will represent its lift under $f_1\circ f_2$ by a holomorphic mapping  to $\mathbb{C}\setminus i\mathbb{Z}$ from a rectangle of extremal length not exceeding a constant times
$\mathcal{L}(w)$.
The properties of the representing mappings will allow quasiconformal gluing.
The building block is the following lemma.

\begin{lemm}\label{lemxxx}
For each elementary word $_{pb}w_{pb}\in \pi_1^{pb}$
there exists
a rectangle $R_w$ with horizontal side length
${\rm hsl}(R_w)\geq \frac{1}{16\sqrt{2}}$,
and vertical side length ${\rm vsl}(R_w)\leq
\pi \mathcal{L}(w)$,
and a holomorphic mapping $g_w: R_w\to
\mathbb{C}\setminus i \mathbb{Z}$ with the following properties.

There are two points $\xi_w^+$ and $\xi_w^{-}$ in the boundary of $R_w$
with ${\rm Re} \xi_w^+={\rm Re} \xi_w^-$,
such that the restriction of $g_w$ to the (vertical) line segment that joins $\xi_w^{-}$ with $\xi_w^{+}$ represents a lift of  $_{pb}w_{pb}$ under $f_1\circ f_2$.
The values of $g_w$ at the points $\xi_w^{\pm}$
are imaginary half-integers that are not imaginary integers.
The equalities for the derivatives at these points are $g_w'(\xi_w^-)= i$ if the first letter of the word is $a_1^{\pm 1}$ and $-i$ if the first letter is $a_2^{\pm 1}$, and  $g_w'(\xi_w^+)= -i$ if the last letter of the word is $a_1^{\pm 1}$ and $i$ if the last letter is $a_2^{\pm 1}$. Further,
the mapping $g_w$ extends holomorphically to the discs $| \zeta-\xi_w^{\pm}|<\frac{1}{16}$ of radius $\frac{1}{16}$ around $\xi_w^{\pm}$ and  the inequality $|g_w''(\zeta)|\leq 2^8$ holds for the extended function if $| \zeta-\xi_w^{\pm}|<\frac{1}{16}$.

For each elementary word $_{tr}w_{pb}$ or  $_{pb}w_{tr}$ of the form $(1)$ or $(3)$  with mixed horizontal boundary values there is a lift under $f_1\circ f_2$ that can be represented by a holomorphic mapping from a  rectangle with horizontal side length
${\rm hsl}(R_w)\geq \frac{1}{16\sqrt{2}}$,
and vertical side length ${\rm vsl}(R_w)\leq
\pi \mathcal{L}(w)$. Moreover, the horizontal side, that is mapped to $i\mathbb{R}$, contains a point for which the previous statements hold.
\end{lemm}
\index{${\rm hsl}$} \index{${\rm vsl}$} \index{$R_w$} \index{$g_w$} \index{$\xi_w^{\pm}$}

\medskip

\noindent {\bf Proof of Lemma \ref{lemxxx}.}\\
{\bf 1. Syllables of form (1) with $pb$ horizontal boundary conditions.} We first consider syllables $_{pb}w_{pb}$ of the form $_{pb}(a_1^d)_{pb},\,$ $ d\geq 2$. Let $M=\frac{d-1}{2}$.
Consider the holomorphic mapping $T^{{\sf a},{\sf b}}$ of Example 1, Section \ref{sec:4.1b} with ${\sf a}=M$ and ${\sf b}=M+1$. Denote the obtained mapping by $T_M$.  Let $t_M= -M+\sqrt{M(M+1)}$ (compare with equality \eqref{eq4.1c}).
The inverse $T_M^{-1}$ provides a conformal mapping
from the half-annulus with center $\frac{1}{2}$ and radii $\frac{1}{2}$ and $\frac{1}{2}-t_M$  to $\{z\in \mathbb{C}: {\rm Im} z>0, {\rm Re} z >0, |z-\frac{M+1}{2}|> \frac{1}{2}\}$, that takes the half-circle with diameter $(0,1)$ to the imaginary half-axis and the half-circle with diameter $(t_M,1-t_M)$ to the half-circle with diameter $(M,M+1)$. The mapping takes $0$ to $0$, $1$ to $\infty$,
$t_M$ to $M$ and $1-t_M$ to $M+1$. The first two conditions imply that the mapping $T_M^{-1}$ has the form $z\to \frac{a}{c}\frac{z}{z-1}$. The last two conditions imply that $(\frac{a}{c})^2=M(M+1)$ and $\frac{a}{c}<0$. This implies that
\begin{align}\label{eqbr.10}
T_M^{-1}(z)=\sqrt{M(M+1)}\frac{z}{1-z}\,,\;\;{\rm and}\; T_M(z)=\frac{z}{\sqrt{M(M+1)}+z}\,.
\end{align}
\index{$T_M$} \index{$t_M$}
Further, the entire mapping
\begin{equation}\label{eqbr.11}
\omega(z)= \frac{1}{2}+\frac{1}{2} e^{i z}
\end{equation}
takes the rectangle $\mathring{R}_M^+$ with vertices $0$, $\pi$, $\pi+i\log\frac{1}{1-2 t_M}$, $i\log\frac{1}{ 1-2 t_M}$  conformally onto the half-annulus with center $\frac{1}{2}$ and radii $\frac{1}{2}$ and $\frac{1}{2}-t_M$. The lower side of the rectangle is mapped onto the larger half-circle, the upper side is mapped onto the smaller half-circle.
Let $\mathring{R}_M\supset \mathring{R}_M^+$ be the rectangle that is symmetric with respect to the real axis, has the same horizontal side length ${\rm hsl}(\mathring{R}_M)={\rm hsl}(\mathring{R}_M^+)=\pi$, and
(see equations \eqref{eq4.100}, \eqref{eq3braids.100} and \eqref{eq3braids.101}) the double vertical side length
\begin{align}\label{eq3braids.5}
{\rm vsl}(\mathring{R}_M)=&2{\rm vsl}(\mathring{R}_M^+)=
2 \log\frac{1}{1-2 t_M}
= 2 \log(\sqrt{M}+\sqrt{M+1})^2\nonumber \\
=& 2\log(d+\sqrt{d^2-1})
\leq 2   \log(2d+\sqrt{4d^2-1}).
\end{align}
The rectangle $\mathring{R}_M$ has extremal length
\begin{align}\label{eqbr.12}
\lambda(\mathring{R}_M)= \frac{2 \log\frac{1}{1-2 t_M}}{\pi}
\leq  \frac{2}{\pi}    \log(2d+\sqrt{4d^2-1})=  \frac{2}{\pi} \mathcal{L}(w)\,.
\end{align}
The mapping $\omega$ takes $\mathring{R}_M$ onto the half-annulus in $\mathbb{C}_+$
with center $\frac{1}{2}$ and radii of the circles equal to $\frac{1}{2}-t_M$ and $\frac{1}{ \frac{1}{2}-t_M}$.
\index{$\mathring{R}_M$}
\index{$\mathring{R}_M^+$}
\index{$R_M'$}
\index{$\mathfrak I$}

Denote by $\mathfrak I$ the mapping, that acts by multiplication with the imaginary unit,  $\mathfrak{I}(z)=i z,\, z\in \mathbb{C}$.  The mapping  $ \mathfrak{I}\circ T_M^{-1}\circ \omega = iT_M^{-1}\circ \omega$ takes $\mathring{R}_M$ conformally onto $\mathbb{C}_{\ell}\setminus \big(\{|\zeta-i(M+\frac{1}{2})|\leq \frac{1}{2}\}\cup\{|\zeta+i(M+\frac{1}{2})|\leq \frac{1}{2}\}\big) $.

Recall that $\tilde{\mathbb{C}}_{\ell}$ is the lift of the left half-plane ${\mathbb{C}}_{\ell}$ to the
first sheet of the universal covering $U$ of  $\mathbb{C}\setminus i \mathbb{Z}$.
We consider the conformal mapping $\varphi_0$  from $\tilde{\mathbb{C}}_{\ell}$ onto
$\mathbb{C}_{\ell}\setminus \bigcup_{k=-\infty}^{\infty} \{|z-i(k+\frac{1}{2})|\leq \frac{1}{2}\}$, whose continuous extension to the boundary takes each interval $(ki,(k+1)i)$ to the half-circle $\mathbb{C}_{\ell}\cap\{|z-i(k+\frac{1}{2})|= \frac{1}{2}\}$ (see Section \ref{sec:3-braids1a}).  Let $\varphi$ \index{$\varphi$}
be the extension of $\varphi_0$ \index{$\varphi_0$} to a conformal mapping of the universal covering $U$ of $\mathbb{C}\setminus i\mathbb{Z}$ onto $\mathbb{C}_{\ell}$ (see Lemma \ref{lemm3}), and let $\varphi^{-1}$ be its inverse. Denote by ${\sf P}$ the covering map from the universal covering $U$ of $\mathbb{C}\setminus i\mathbb{Z}$ to the set $\mathbb{C}\setminus i\mathbb{Z}$.
The composition $({\sf P}\varphi^{-1})$ of ${\sf P}$ with the inverse $\varphi^{-1}$
of $\varphi$ maps  $\mathbb{C}_{\ell}$ locally conformally to  $\mathbb{C}\setminus i\mathbb{Z}$. Moreover, it maps the subset $\mathfrak{H}_0\setminus \{|z-\frac{i}{2}|<\frac{1}{2}\}$ of the half-strip $\mathfrak{H}_0\stackrel{def}=\{z\in \mathbb{C}_{\ell}:0<{\rm Im}z<1\}$ \index{$\mathfrak{H}_0$} \index{$\mathfrak{H}$} conformally onto $\mathfrak{H}_0$. Hence, by reflection in the
half-circle $ \{z\in \mathbb{C}_{\ell}:|z-\frac{i}{2}|<\frac{1}{2}\}$ and in the   interval $(0,i)$, respectively, it follows that  $({\sf P}\varphi^{-1})$ maps
 $Q\stackrel{def}=\mathfrak{H}_0\setminus \big(\{|z-\frac{i}{4}|<\frac{1}{4}\}\cup \{|z-\frac{3i}{4}|<\frac{1}{4}\}$
conformally onto the strip $\mathfrak{H}=\{z\in \mathbb{C}: 0<{\rm Im}z<1\}$.
The mapping ${\sf P}\circ \varphi^{-1}\mid Q$ takes $\frac{i-1}{2}$ to $\frac{i}{2}$.

The mapping $({\sf P}\varphi^{-1})_M,\, ({\sf P}\varphi^{-1})_M(z)\stackrel{def}={\sf P}\varphi^{-1}(z-iM)+iM$ takes $\mathbb{C}_{\ell}\setminus \{|z-i(M\pm\frac{1}{2})|< \frac{1}{2}\}$ locally conformally into $\mathbb{C}\setminus ( i\mathbb{Z}+iM)$ so that its continuous extension takes the half-circle $\mathbb{C}_{\ell}\cap\{|z-i(M+\frac{1}{2})|= \frac{1}{2}\}$ to  $(iM,i(M+1))$, and the half-circle $\mathbb{C}_{\ell}\cap\{|z+i(M+\frac{1}{2})|= \frac{1}{2}\}$ to $(-i(M+1),-iM)$.
The composition
\begin{align}\label{eq10.100}
\mathcal{G}_M\stackrel{def}= ({\sf P}\circ \varphi^{-1})_M\circ \mathfrak{I}\circ T_M^{-1}\circ\omega
\end{align}
maps $\mathring{R}_M$ locally conformally into $\mathbb{C}\setminus ( i\mathbb{Z}+iM)$, and its continuous extension takes the lower side of $\mathring{R}_M$ to $(-(M+1)i, -Mi)$ and the upper side to the interval $(iM,i(M+1))$. See Figure \ref{Fig3-braids.9}.

\begin{figure}[h]
\begin{center}
\includegraphics[width=11.5cm]{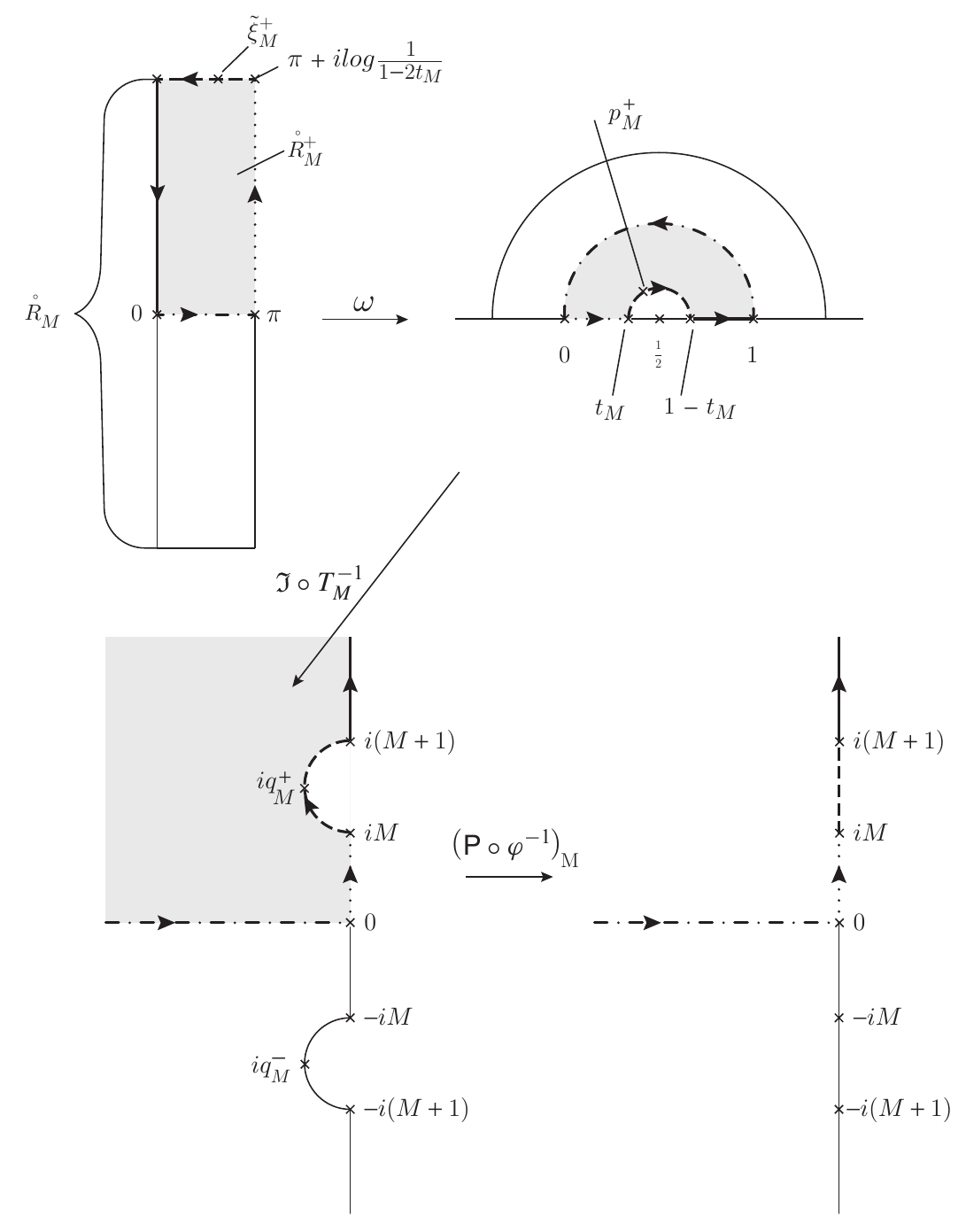}
\end{center}
\caption{A conformal mapping representing a lift of a syllable of form (1).}
\label{Fig3-braids.9}
\end{figure}

The mapping $\mathcal{G}_M -iM:\mathring{R}_M\to \mathbb{C}\setminus i\mathbb{Z}$ represents a lift under $\mathcal{F}\stackrel{def}=f_1\circ f_2$ of $_{pb} (a_1^d)_{pb}$ (see Section \ref{sec:3-braids1a}, and also Figure \ref{Fig3-braids1}). In other words, $\mathcal{F} \circ (\mathcal{G}_M-iM)$ represents
$_{pb} (a_1^d)_{pb}$. Let $R'_M\supset \mathring{R}_M$ be the rectangle of vertical side length $3 \cdot {\rm vsl}(\mathring{R}_M)$ and horizontal side length $ {\rm hsl}(\mathring{R}_M)$ which is symmetric with respect to the real axis. By the reflection principle $\mathcal{G}_M$ extends holomorphically through each of the horizontal sides of $\mathring{R}$ to a holomorphic mapping from  the rectangle $R'_M$
into $\mathbb{C}\setminus i\mathbb{Z}$. The extension is denoted again by $\mathcal{G}_M$. (see Figure \ref{Fig3-braids.9}.)
\index{$\mathcal{G}_M$} \index{$\mathfrak{F}$}

We put $q_M^{\pm}\stackrel{def}=\pm\frac{2M+1}{2}+\frac{i}{2}$. Notice that $iq_M^{\pm}$ is the midpoint of the half-circle in the left half-plane with diameter $(iM,i(M+1)$ ($(-i(M+1),-iM)$, respectively).
Let $\tilde{\xi}_M^{\pm}$ and $p_M^{\pm}$ be the points for which
\begin{align*}
\tilde{\xi}_M^{\pm} \overset {\omega} {-\!\!\!-\!\!\!\longrightarrow}   p_M^{\pm}\overset{\mathfrak{I}\circ T_M^{-1}}   {-\!\!\!-\!\!\!-\!\!\!\longrightarrow}    i q_M^{\pm}\,.
\end{align*}
\index{$q_M^{\pm}$}  \index{$\tilde\xi_M^{\pm}$} \index{$p_M^{\pm}$}

We will consider the inverse of $\mathfrak{I}\circ T_M^{-1}\circ \omega$,
$  (\mathfrak{I}\circ  T_M^{-1}\circ \omega)^{-1}= {\omega}^{-1}\circ T_M\circ \mathfrak{I}^{-1}$.

By equality \eqref{eqbr.10}, and since for the branch of the logarithm on $\mathbb{C}_+$ with imaginary part in $(0,\pi)$ the equality $\omega^{-1}(z)=\frac{1}{i} \log(2 z-1)$ holds, we obtain
\begin{align}\label{eq3braids6}
\omega^{-1}\circ T_M(z)= \frac{1}{i} \log(2T_M(z)-1)= \frac{1}{i} \log\frac{-\sqrt{M(M+1)} +z}{+\sqrt{M(M+1)} +z}\,.
\end{align}
For the derivative we get
\begin{align}\label{eqbr.30}
(\omega^{-1}\circ T_M)'(z)=& \frac{1}{i}\left(\frac{1}{-\sqrt{M(M+1)}+z} -\frac{1}{+\sqrt{M(M+1)}+z}\right)\nonumber\\
=&i\frac{2\sqrt{M(M+1)}}{M(M+1)-z^2}\,.
\end{align}
Hence,
\begin{align*}
(T_M^{-1}\circ\omega)'(\tilde{\xi}_M^+)=&\frac{1}{(\omega_M^{-1}\circ T_M)'(M+\frac{1}{2}+\frac{i}{2})}\nonumber\\
=&\frac{1}{2i} \frac{M(M+1)-(M+\frac{1}{2}+\frac{i}{2})^2}{\sqrt{M(M+1)}}=-\frac{1}{2}\frac{M+\frac{1}{2}}{\sqrt{M(M+1)}}\,.
\end{align*}
By equation \eqref{eq3braids.102} the equality $({\sf P}\circ \varphi_0^{-1})'(z+i)=({\sf P}\circ \varphi_0^{-1})'(z)$ holds on ${\mathbb{C}_{\ell}}\setminus \bigcup_{k=-\infty}^{\infty} \{|z-i(k+\frac{1}{2})|\leq \frac{1}{2}\}$, hence $({\sf P}\circ \varphi^{-1})'(z+i)=({\sf P}\circ \varphi^{-1})'(z)$ holds for $z\in {\mathbb{C}_{\ell}}$. This implies that
\begin{align}\label{eqbr.19}
(\mathcal{G}_M)'(\tilde\xi_M^+)=&\big(({\sf P}\circ \varphi^{-1})_M\circ (i T_M^{-1})\circ \omega\big)'(\tilde\xi_M^+)=\nonumber \\
&({\sf P}\circ \varphi^{-1})_M'(iq_M^+)\cdot
(i T_M^{-1}\circ \omega)'(\tilde\xi_M^+) =\nonumber \\ &({\sf P}\circ \varphi^{-1})'(\frac{i-1}{2})\cdot  \left(-\frac{i}{2}\frac{M+\frac{1}{2}}{\sqrt{M(M+1)}}\right)\,.
\end{align}
We used that $iq_M^+ + iM=-\frac{1}{2}+i(M+\frac{1}{2}) +iM$ differs by an imaginary integer from $-\frac{1}{2}+\frac{1}{2}i$.

The derivative $({\sf P}\circ \varphi^{-1})'(\frac{i-1}{2})$
is a positive real number. Indeed, the composition ${\sf P}\circ \varphi^{-1} $ maps $Q=\mathfrak{H}_0\setminus \big(\{|z-\frac{i}{4}|<\frac{1}{4}\}\cup \{|z-\frac{3i}{4}|<\frac{1}{4}\}$ conformally onto $\mathfrak{H}_0$, and by the symmetry properties of the source and the target it maps the horizontal line $\{{\rm Im}z=\frac{1}{2}\}$ to itself, and its real part is increasing in $x={\rm Re}z$.
Hence, $\frac{\partial}{\partial z }({\sf P}\circ \varphi^{-1})(\frac{i-1}{2})=
\frac{\partial}{\partial x }({\sf P}\circ \varphi^{-1})(\frac{i-1}{2})$ is a positive real number $\alpha_{\varphi}$ that does not depend on $M$.

We will prove now, that $\alpha_{\varphi}\leq \frac{8}{\pi}$.
The set $Q$ contains the disc of radius $\frac{1}{4}$ around the point $\frac{i}{2}-\frac{1}{2}$. The mapping $\varrho,\, \varrho(z)=
\frac{e^{-\pi  z}i-1}{e^{- \pi  z}i +1}$, is a conformal mapping from
the strip $\mathfrak{H}$ onto the unit disc. It maps $\frac{i}{2}$ to $0$. Hence the composition $\varrho\circ{\sf P}\circ \varphi^{-1}:Q\to \mathbb{D}$ maps the disc of radius $\frac{1}{4}$ around the point $\frac{i}{2}-\frac{1}{2}$ conformally onto a subset of the unit disc.
\index{$\varrho$} \index{$Q$}

By Cauchy's formula the derivative of the composition $\varrho\circ{\sf P}\circ \varphi^{-1}$ at
$\frac{i}{2}-\frac{1}{2}$ does not exceed $4$. The derivative of $\varrho$ at $\frac{i}{2}$ equals  $\varrho'(\frac{i}{2})=-\frac{\pi}{2} $. Since
$(\varrho\circ{\sf P}\circ \varphi^{-1})'(\frac{i}{2}-\frac{1}{2})=\varrho'(\frac{i}{2})\cdot({\sf P}\circ \varphi^{-1})'(\frac{i}{2}-\frac{1}{2}) $, the inequality
$|({\sf P}\circ \varphi^{-1})'(\frac{i}{2}-\frac{1}{2})|\leq \frac{8}{\pi}$ holds.

We proved that
\begin{align}\label{eqbr.20}
(\mathcal{G}_M)'(\tilde\xi_M^+)= -i \cdot
\frac{\alpha_{\varphi} }{2}\frac{M+\frac{1}{2}}{\sqrt{M(M+1)}}\,
\end{align}
with $0<\alpha_{\varphi}\leq \frac{8}{\pi}$. \index{$\alpha_{\varphi}$}  By symmetry reasons
\begin{align}\label{eqbr.20a}
(\mathcal{G}_M)'(\tilde\xi_M^-)= i \cdot
\frac{\alpha_{\varphi} }{2}\frac{M+\frac{1}{2}}{\sqrt{M(M+1)}}\,.
\end{align}
Since
\begin{align*}
(\frac{M+\frac{1}{2}}{\sqrt{M(M+1)}})^2=1+\frac{1}{4M(M+1)}\leq \frac{4}{3}\; \mbox{for}\; M\geq \frac{1}{2}\,,
\end{align*}
we get
\begin{align}\label{eq3-braids.2}
|(\mathcal{G}_M)'(\tilde\xi_M^{\pm})|\leq \frac{8}{\sqrt{3}\pi}\,.
\end{align}
We put \index{$R_M$} \index{$g_M$} \index{$\xi_M^{\pm}$}
\begin{align}\label{eq3-braids.3}
R_M=& |(\mathcal{G}_M)'(\tilde{\xi}_M^+)|\cdot \mathring{R}_M\,,\nonumber\\
\xi_M^{\pm}=&{|(\mathcal{G}_M)'(\tilde{\xi}_M^{+})|}\cdot{{\tilde \xi}_M^{\pm}}\,\nonumber,\\
g_M(\zeta)=&(\mathcal{G}_M)\big(\frac{\zeta}{|(\mathcal{G}_M)'(\tilde{\xi}_M^{+})|}\big),\, \zeta \in R_M\,.
\end{align}
By equations \eqref{eqbr.20} and \eqref{eqbr.20a} the equality $g_M'(\xi_M^{\pm})=\mp i$ holds.

We will prove now that $g_M$ extends holomorphically to
a disc of radius $\frac{1}{8}$ around $\xi_M^{\pm}$, and for the extended function the inequality $|g_M''(\zeta)|<2^8$
holds on a disc of radius $\frac{1}{16}$ around  $\xi_M^{\pm}$. Let $R_{M,+}$ be the rectangle of twice the vertical
side length and the same horizontal side length as $R_{M}$, that is symmetric with respect to the upper side of $R_M$.
The mapping $g_M$ takes $R_M$ locally conformally to $\mathbb{C}\setminus (i\mathbb{Z}+iM)$ and its continuous extension to the boundary maps the upper side to
the segment $(iM, i(M+1))$. Hence, $g_M$ extends by reflection through the upper side of the rectangle and the segment $(iM,i(M+1))$, respectively, to a locally conformal mapping from $R_{M,+}$ into  $\mathbb{C}\setminus (i\mathbb{Z}+iM)$.
Denote the extension again by $g_M$. Put $V= |(\mathcal{G}_M)'(\tilde{\xi}_M^+)|\cdot (\mathfrak{I}\circ T_M^{-1}\circ \omega)^{-1}(Q)$. Note that $V\subset R_{M,+}$. Since ${\sf P}\circ \varphi^{-1}\mid Q$ is a conformal mapping, and ${\mathcal G}_M= {\sf P}\circ \varphi^{-1}\circ (\mathfrak{I}\circ T_M^{-1}\circ \omega -iM)+iM$, the restriction $g_M\mid V$ is by \eqref{eq3-braids.3} a conformal mapping onto a subset of $\mathfrak{H}+iM$. The
image of the mapping $g_M\mid V$ contains  the disc of radius $\frac{1}{2}$ around $i(M+\frac{1}{2})$. Let $\mathfrak{f}_M$ be  the inverse of  $g_M\mid V$.

The restriction of $\mathfrak{f}_M$ to the disc of radius $\frac{1}{2}$ around $i(M+\frac{1}{2})$ is a conformal mapping with image in   $R_{M,+}$.
The mapping \index{$R_{M,+}$}
$\zeta\to 2\mathfrak{f}_M(i\frac{M+1}{2}+\frac{\zeta}{2})$ takes the unit disc $\mathbb{D}$ conformally onto its image.
It takes $0$ to $\xi_M^+$, and has derivative of absolute value equal to $1$ at $0$. By Koebe's $\frac{1}{4}$ Covering Theorem (see \ref{thmKoebe}) \index{Koebe ! $\frac{1}{4}$ Covering Theorem} the image of the unit disc under this mapping contains the disc of radius $\frac{1}{4}$ around $\xi_M^+$. Hence, the image $\mathfrak{f}_M(\{|z-i\frac{M+1}{2}|<\frac{1}{2}\})$ covers the disc of radius $\frac{1}{8}$ around $\xi_M^+$.
In particular, the disc of radius $\frac{1}{8}$
around $\xi_M^+$ is contained in $R_{M,+}$, hence ${\rm hsl}(R_M)\geq \frac{1}{4}$.
Moreover, the mapping  $g_M$, being the inverse of $\mathfrak{f}_M$, takes a disc of radius  $\frac{1}{8}$ around $\xi_M^+$ into a disc of radius  $\frac{1}{2}$.
By Cauchy's formula
the second derivative of $g_M$ on the disc of radius $\frac{1}{16}$
around $\xi_M^+$ is estimated by the inequality $|g_M''(\zeta)|\leq 2^8$.

The same arguments apply to the point $\xi^-_M$ instead of $\xi_M^+$ obtained from
$\xi_M^+$ by reflection in the real axis. We obtain
\begin{align}\label{eqbr.21}
|g_M''(\zeta)|\leq 2^8 \;\mbox{for}\; |\zeta-\xi_M^{\pm}| <\frac{1}{16}  \,.
\end{align}

For $w=a_1^d$ and $M=\frac{d-1}{2}$ we put  $\xi_w^{\pm}=\xi_M^{\pm}-\frac{\xi_M^+ +\xi_M^-}{2}$. The rectangle $R_M-\frac{\xi_M^+ +\xi_M^-}{2} $ is symmetric with respect to the imaginary axis and with respect to the real axis..
We define
$R_w$ to be the intersection of $R_M-\frac{\xi_M^+ +\xi_M^-}{2} $ with the vertical strip $\{|{\rm Re}(z)|<\frac{1}{8}\}$.  The points $\xi_w^{\pm}$ are the midpoints of the horizontal sides of $R_w$.

Finally we put $g_w(\zeta)=g_M(\zeta+\frac{\xi_M^+ +\xi_M^-}{2})+iM$.
The horizontal boundary values of the mapping $g_w$ are contained in the intervals $(-i,0)$ and $(2Mi,(2M+1)i)
=((d-1)i,di)$.
Hence $g_w$ represents a lift of $a_1^d$ under $f_1\circ f_2$.
By \eqref{eq3braids.5},
\eqref{eq3-braids.2}, and \eqref{eq3-braids.3} the vertical side length of $R_w$ is estimated by
\begin{align*}
\mbox{vsl}(R_w)=  \mbox{vsl}(R_M)\leq \frac{16}{\sqrt{3}\pi}\mathcal{L}(w)<\pi \mathcal{L}(w)\,.
\end{align*}
The lemma is proved for syllables of the form $a_1^d$ with $d\geq 2$.

The statement for
syllables of the form $a_1^{-d}$ and $a_2^{\pm d}$ with $d\geq 2$ follows by symmetry. Indeed, the mapping $z\to -z$ takes $R_{a_1^d}$, homeomorphically onto itself, and takes $\mathbb{C}\setminus i\mathbb{Z}$ homeomorphically onto itself.
The mapping $z\to -g_{a_1^d}(z),\, z\in R_{a_1^d}$, represents a lift of  $a_2^d$ under $f_1\circ f_2$. The mapping $z\to g_{a_1^d}(-z),\, z\in R_{a_1^d}$, represents a lift of $a_1^{-d}$, and the mapping $z\to -g_{a_1^d}(-z),\, z\in R_{a_1^d}$ represents a lift of $a_2^{-d}$ under $f_1\circ f_2$. The lemma is proved for syllables of the form (1) with $pb$ horizontal boundary conditions.
\index{$R_w$} \index{$g_w$} \index{$\xi_w^{\pm}$}

\medskip

\noindent {\bf 2.  Syllables of form (1) with mixed horizontal boundary conditions.}\\
Consider the elementary word $_{tr}(a_1^d)_{pb}$ with mixed horizontal boundary values. The other cases are similar. We represent this elementary word as follows.
Put $M'=d-\frac{1}{2}$. The mapping $(\mathcal{G}_{M'}-\frac{i}{2})\mid \mathring{R}_{M'}^+$ represents a lift under $f_1\circ f_2$ of the word $_{tr}(a_1^d)_{pb}$. (Indeed, the restriction to vertical line segments join $-\frac{i}{2} +\mathbb{R}$ with $(i(M'-\frac{1}{2}),  i(M'+\frac{1}{2}))$.) The extremal length of the rectangle is
\begin{align*}
\lambda (\mathring{R}_{M'}^+)=&\frac{1}{\pi} \log(\sqrt{M'} +\sqrt{M' +1})^2
=\frac{1}{\pi} \log(2 M'+1+2\sqrt{M'(M'+1)})\\ =&\frac{1}{\pi} \log(2d+2 \sqrt{(d-\frac{1}{2})(d+\frac{1}{2}})=\frac{1}{\pi}  \log(2d+\sqrt{4d^2-1})=\frac{1}{\pi}\mathcal{L}(w)\,.
\end{align*}
The rest of the proof is the same as above.

\medskip

\noindent{\bf 3.  Singletons.} We consider syllables of degree $d=1$. We start with the singleton $w=_{pb}(a_2^{-1})_{pb}$. A lift under $f_1\circ f_2 $ is represented by the mapping $g_0(\zeta)\stackrel{def}=\frac{1}{2} e^{2 \zeta}$ from the rectangle $\mathring{R}_0$ with vertices  $\pm \log 2 \pm \frac{\pi}{4} i$
into $\mathbb{C}\setminus i\mathbb{Z}$. The points ${\xi}_0^{\pm}= \pm i\frac{\pi}{4}$ are mapped to $\pm \frac{i}{2}$. The derivative of the mapping at these points equals $(g_0)'({\xi}_0^{\pm})=\pm i $.Moreover, $|g_0''(\zeta)|\leq 8$, $\mbox{hsl}(\mathring{R}_0)=\log 4 >\frac{\sqrt{2}}{16}$, $ \mathcal{L}(w)=\log(2d+\sqrt{4d^2-1})=\log(2+\sqrt{3})$, and $\mbox{vsl}(\mathring{R}_0)= \frac{\pi}{2}  =\frac{\pi}{2\log(2+\sqrt{3})}\mathcal{L}(w)<\pi \mathcal{L}(w)$.
For $w=a_2^{-1}$ we put $R_w=\mathring{R}_0  $, $\xi_w^{\pm}=\xi_0^{\pm}$,
and $g_w=g_0$. The points  $\xi_w^{\pm}$ are the midpoints of the horizontal sides of $R_w$. The lemma is proved for $_{pb}(a_2^{-1})_{pb}$.

For the syllables $a_2$ and $a_1^{\pm 1}$ with $pb$ horizontal boundary values the statement is obtained by the symmetry arguments used for syllables of form $(1)$.
\index{$\mathcal{G}_0$} \index{$\mathring{R}_0$} \index{$g_0$} \index{$R_0$} \index{$\xi_0^{\pm}$} \index{$\tilde{\xi}_0^{\pm}$}
\index{$g_w$} \index{$R_w$} \index{$\xi_w^{\pm 1}$}

Consider a singleton with mixed boundary values. For instance, a lift under $f_1\circ f_2$ of the word $_{tr}(a_2^{-1})_{pb}$ is represented by the mapping $ \tilde{g}(\zeta)=-\frac{i}{2} +e^{\zeta},\, \zeta \in R_0^{\rm mixed}$,
where $R_0^{\rm mixed}\stackrel{def}=
\{\zeta \in \mathbb{C}: {\rm Im}\zeta \in (0,\frac{\pi}{2}),\, |{\rm Re}\zeta| <\frac{1}{5}\}$.
 Indeed, the lower side of the rectangle is mapped to $-\frac{i}{2} +\mathbb{R}$, and the upper side is mapped to $(0,i)$.
Further, $g'(\frac{\pi}{2}i)=i$, $|g''|< e^{\frac{1}{5}}<
\frac{3}{2}$,
${\rm hsl}(R_0^{\rm mixed})=\frac{2}{5}>\frac{\sqrt{2}}{16}$,
and ${\rm vsl}(R_0^{\rm mixed})=\frac{\pi}{2}$.  The other cases with mixed horizontal boundary values follow by symmetry arguments.

\medskip
\noindent {\bf 4. Syllables of form (2) with $pb$ horizontal boundary values of degree at least $2$.}
We give the proof for the syllables $w=_{pb}(a_2^{-1} a_1^{-1} \ldots )_{pb}$. For other syllables of form (2) and degree $d$ the statement follows by symmetry reasons. Put $M=\frac{d+1}{2}$, where $d\geq 2$ is the degree, i.e. the number of letters of the word.
The holomorphic function \index{$\mathfrak{g}_M$}
\begin{align}\label{eq50}
\mathfrak{g}_M(z)=\frac{1}{2}\exp(\pi(z+i(M-1)),\, z \in \mathbb{C}\,,
\end{align}
takes the set $\{z\in\mathbb{C}: {\rm Re}z<\frac{\log 2}{\pi},\,|{\rm Im}z|<M\}$
to the punctured unit disc $\{0<|\zeta|<1\} \subset \mathbb{C}\setminus i\mathbb{Z}$.
The mapping $\mathfrak{g}_M$ takes $\{{\rm Im}z=-(M-\frac{1}{2}), \,{\rm Re}z<\frac{\log 2}{\pi} \}$ to the
line segment $(-i,0)$. Further,  $\mathfrak{g}_M$ takes $\{{\rm Im}z=M-\frac{1}{2},\, {\rm Re}z<\frac{\log 2}{\pi}\}$ to the line segment
$(0,i)$ if $d=2M-1$ is odd, and takes it to $(-i,0)$ if $d=2M-1$ is even.
The restriction of $\mathfrak{g}_M$ to $(-i(M-\frac{1}{2}),i(M-\frac{1}{2}))$ is the curve
$  \theta \to {\frac{1}{2}}\cdot e^{\pi i \theta}, \theta \in (-\frac{1}{2},d-\frac{1}{2})$.
The initial point of this curve is $-\frac{i}{2}$ and the curve makes $d$ half-turns around $0$. The curve $  \theta \to   {\frac{1}{2}}\cdot e^{\pi i \theta},\,\theta \in (-\frac{1}{2},d-\frac{1}{2}),$ represents a lift under $f_1 f_2$ of the syllable $a_2^{-1} a_1^{-1} \ldots$ of degree $d$ with $pb$ boundary values, and the restriction $\mathfrak{g}_M\mid \{z\in \mathbb{C}:{\rm Re}z<\frac{\log 2}{\pi} ,\, |{\rm Im}z|<M-\frac{1}{2}\}$ represents a lift of  $_{pb}(a_2^{-1} a_1^{-1} \ldots )_{pb}$ under $f_1 f_2$.
The derivative of the mapping $\mathfrak{g}_M$ at $- i(M-\frac{1}{2})$ equals
\begin{align}\label{eq51}
(\mathfrak{g}_M)'( -i(M-\frac{1}{2}))=\pi \mathfrak{g}_M( -i(M-\frac{1}{2})) = -\frac{\pi i}{2}\,.
\end{align}
and
\begin{equation}\label{eq52}
(\mathfrak{g}_M)'( +i(M-\frac{1}{2}))=
\begin{cases}
+\frac{\pi i}{2} ,\, & \mbox{if}\, d=2M-1 \, \mbox{is\, odd}\,,\\
-\frac{\pi i}{2}  ,\, & \mbox{if}\, d=2M-1 \, \mbox{is\, even}\,.
\end{cases}
\end{equation}

For each natural number $k$
the set  $\{z\in\mathbb{C}: {\rm Re}z<\frac{\log 2}{\pi}   ,\, |{\rm Im}z|<M- \frac{1}{2}\}$ contains the rectangle  $\mathcal{R}_k\stackrel{def}=\{z\in\mathbb{C}: -k<{\rm Re}z<\frac{\log 2}{\pi},\, |{\rm Im}z|<M-\frac{1}{2}\}$ \index{$\mathcal{R}_k$}
of extremal length $\frac{2M-1}{k+\frac{\log 2}{\pi}}$, which can be made
arbitrarily small  by choosing large $k$. The restriction of $\mathfrak{g}_M$
to each of these rectangles represents a lift of  $_{pb}(a_2^{-1} a_1^{-1} \ldots )_{pb}$ under $f_1 f_2$.
However, we need to represent lifts of syllables of form (2) with parameter $M$ by holomorphic mappings from
rectangles with horizontal side length uniformly bounded from above and from below for all $M$, with vertical side length comparable to $\log M$, and the values of the derivatives of the mappings at the chosen points equal to $\pm i$ where the sign is prescribed by the situation. After normalizing the rectangle and the mapping so that the condition for the derivatives of the mapping is satisfied, the described approach gives rectangles of vertical side length comparable to $d=2M-1$, which does not give the optimal estimate.

For dealing with this difficulty, we precompose $\mathfrak{g}$ with the mapping used for syllables of form $(1)$.
We consider the rectangle $\mathring{R}_M$ \index{$\mathring{R}_M$} and the mapping $\mathfrak{I}\circ T_M^{-1}\circ \omega$ on $\mathring{R}_M$ as in the case of syllables of form (1). Recall that
\begin{align*}
\mbox{vsl}(\mathring{R}_M)=2\log(\sqrt{M}+\sqrt{M+1})^2 
\leq 2\log(2d+\sqrt{4d^2-1})=2\mathcal{L}(w)\,.
\end{align*}
Let $\tilde{R}_M$ be the union of $\mathring{R}_M$ with its open right side and with its reflection through the line $\{{\rm Re} z= \pi\}$. The mapping $\mathfrak{I}\circ T_M^{-1}\circ \omega$ extends holomorphically to $\tilde{R}_M$
by Schwarz Reflection Principle.
Denote the extended mapping again by $\mathfrak{I}\circ T_M^{-1}\circ \omega$. The extended mapping  $\mathfrak{I}\circ T_M^{-1}\circ \omega$    takes $\tilde{R}_M$ conformally onto the complex plane with the discs of radius $\frac{1}{2}$ around $\pm i(M+\frac{1}{2})$ and the rays $(-\infty,-  i(M+1)]$ and
$ [i(M+1), \infty)$ removed.
Put \index{$\tilde{R}_M$}
\begin{align*}
\mathring{R}_{M,2}\stackrel{def}=&(\mathfrak{I}\circ T_M^{-1}\circ \omega)^{-1}\Big(\big\{z\in\mathbb{C}_{\ell}:  |\mbox{Im}z |<M-\frac{1}{2}\big\}\Big)\,,\\
R'_{M,2}\stackrel{def}=&(\mathfrak{I}\circ T_M^{-1}\circ \omega)^{-1}\Big(\big\{z\in\mathbb{C}: {\rm Re z}<\frac{\log 2}{\pi},\,  |\mbox{Im}z |<M\big\}\Big)\,.
\end{align*}
\index{$\mathring{R}_{M,2}$} \index{$R'_{M,2}$}
Then $\mathring{R}_{M,2}\subset \mathring{R}_{M}\subset \tilde{R}_M$, and
$\mathfrak{I}\circ  T_M^{-1}\circ \omega$ takes $\mathring{R}_{M,2}$ conformally onto $\{z\in\mathbb{C}_{\ell}:  |\mbox{Im}z |<M-\frac{1}{2}\}$, and takes $R'_{M,2}$ conformally onto
$\{z\in\mathbb{C}: \mbox{Re}z<\frac{\log 2}{\pi} ,\, |\mbox{Im}z |<M\}$. The composition
$\mathcal{G}_M\stackrel{def}= \mathfrak{g}_M \circ \mathfrak{I}\circ  T_M^{-1}\circ \omega$ takes $R'_{M,2}$ into the punctured disc
$\{0<|z|<1\}$. \index{$\mathcal{G}_M$}
We put
$$
\tilde\eta_M^{\pm}\stackrel{def}=(\mathfrak{I}\circ  T_M^{-1}\circ \omega)^{-1}(\pm i(M-\frac{1}{2}))=(T_M^{-1}\circ \omega)^{-1}(\pm (M-\frac{1}{2}))\,.
$$
The points $\tilde\eta_M^{\pm} $ lie on the boundary of $\mathring{R}_{M,2}$ and on the open right side of $\mathring{R}_{M}$ (which is contained in $\{{\rm Re}z=\pi\}\cap R'_{M,2}$).

We prove now that the set $\mathring{R}_{M,2}$ contains the set
\begin{align}\label{eqbr.28}
\widehat{R}_M\stackrel{def}=\big\{z\in \mathbb{C}:-  \frac{2}{5}<    {\rm Re}(z-\pi)<0,\; {\rm Im}(z-\tilde{\eta} _M^+)<&{\rm Re} (z-\pi),\,
\nonumber\\
{\rm Im}(z-\tilde{\eta} _M^-)>&-{\rm Re} (z-\pi)\big\}\,.
\end{align}
\index{$\widehat{R}_M$}
We describe first the curves
\begin{align*}
\Gamma^{\pm}\stackrel{def}=(\mathfrak{I}\circ  T_M^{-1}\circ\omega)^{-1}\Big(\big\{z\in \mathbb{C}: \,  -  \frac{2}{5}<    {\rm Re}z<0   ,\,{\rm Im} z=\pm(M-\frac{1}{2})\big\}\Big)
\end{align*}

\begin{figure}[h]
\begin{center}
\includegraphics[width=12.5cm]
{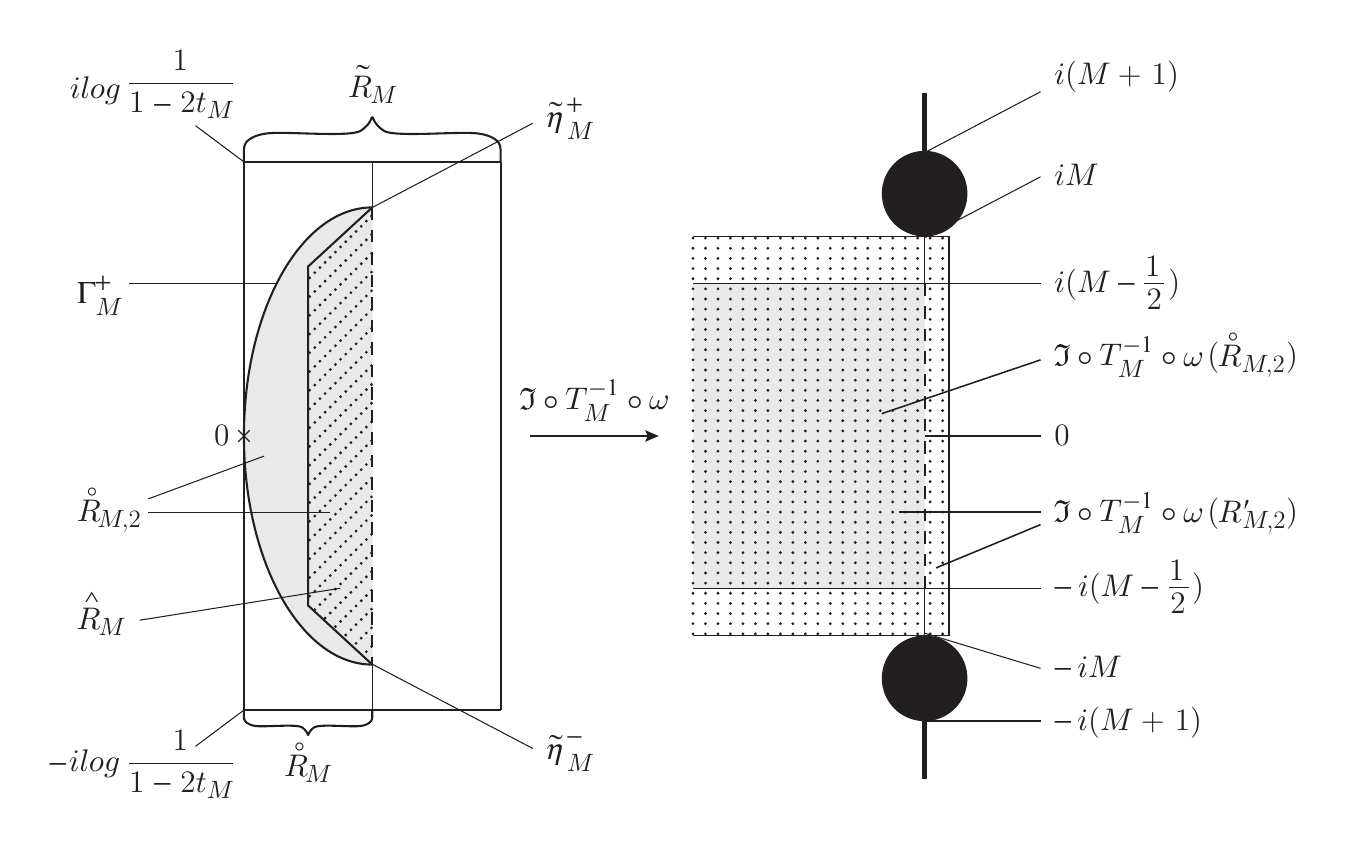}
\end{center}
\caption{The construction of a rectangle that admits a holomorphic mapping representing a syllable of form (2). }
\label{Fig3-braids.10}
\end{figure}

Since $(\mathfrak{I}\circ  T_M^{-1}\circ \omega)^{-1}= {\omega}^{-1}\circ T_M\circ \mathfrak{I}^{-1}$, we may also write
\begin{align*}
\Gamma^{\pm}=\omega^{-1}\circ T_M \Big(\{z\in \mathbb{C}: \,0<{\rm Im} z < \frac{2}{5}   ,\,{\rm Re} z=\pm (M-\frac{1}{2})\}\Big)\,.
\end{align*}
To describe the curve $\Gamma^{+}$,
we consider the derivative (see equation \eqref{eqbr.30})
\begin{align}\label{eqbr.29}
\frac{\partial}{\partial t} (\omega^{-1}\circ T_M)(M-\frac{1}{2} +it)= &
i(\omega^{-1}\circ T_M)'(M-\frac{1}{2} +it)\nonumber\\
=&\frac{2\sqrt{M(M+1)}}{(M-\frac{1}{2}+ it)^2-M(M+1)}\,,\, t\geq 0\,.
\end{align}
(We used that for a holomorphic function $f$ the equality $\frac{\partial}{z}f=-i\frac{\partial}{y}f$ holds.)
The denominator of the last fraction equals
$-2M -t^2 +\frac{1}{4} +2t(M-\frac{1}{2})i$. Hence, (see also equation \eqref{eq3braids6})
\begin{align}\label{eqbr.27}
{\rm Re}\frac{\partial}{\partial t} \frac{1}{i}\log(2T_M(M-\frac{1}{2} +it)-1)&= \frac{2\sqrt{M(M+1)}(-2M -t^2 +\frac{1}{4})     }
{(-2M -t^2 +\frac{1}{4})^2+(2t(M-\frac{1}{2}))^2}<0 \,,\nonumber
\\
{\rm Im}\frac{\partial}{\partial t} \frac{1}{i}\log(2T_M(M-\frac{1}{2} +it)-1)&=  \frac{-2\sqrt{M(M+1)}2t(M-\frac{1}{2})}
{(-2M -t^2 +\frac{1}{4})^2+(2t(M-\frac{1}{2}))^2}<0\,.
\end{align}
We obtain
\begin{align*}
\frac{{\rm Im}\frac{\partial}{\partial t} \frac{1}{i}\log(2T_M(M-\frac{1}{2} +it)-1)}{{\rm Re}\frac{\partial}{\partial t} \frac{1}{i}\log(2T_M(M-\frac{1}{2} +it)-1)}
=\frac{2t(M-\frac{1}{2})}{2M+t^2-\frac{1}{4}}>0\,.
\end{align*}
For $0\leq t\leq \frac{1}{2}$
\begin{align}\label{eq53}
\frac{{\rm Im}\frac{\partial}{\partial t} \frac{1}{i}\log(2T_M(M-\frac{1}{2} +it)-1)}{{\rm Re}\frac{\partial}{\partial t} \frac{1}{i}\log(2T_M(M-\frac{1}{2} +it)-1)}
\leq \frac{M-\frac{1}{2}}{2M-\frac{1}{4}}<\frac{1}{2}\,.
\end{align}
Further, since $M\geq \frac{3}{2}$, the following estimate holds for $0\leq t\leq \frac{1}{2}$
\begin{align}\label{eq54}
\Bigl|
{\rm Re}\frac{\partial}{\partial t} \frac{1}{i}\log(2T_M(M-\frac{1}{2} +it)-1)
\Bigl|\geq &\frac{2M(2M -\frac{1}{4})}
 {(2M)^2+(M-\frac{1}{2})^2}\nonumber \\
 =&\frac{4M^2-\frac{M}{2}}{5M^2-M+\frac{1}{4}}>\frac{4}{5}\,.
\end{align}
By inequality \eqref{eq54} for $t=\frac{1}{2}$ the point $\frac{1}{i}\log(2T_M(M-\frac{1}{2} +it)-1)$ lies in $\{z\in \mathbb{C}: {\rm Re}z<\pi -\frac{2}{5}\}$.
Recall that $\frac{1}{i}\log(2T_M(M-\frac{1}{2})-1)=\tilde{\eta}_M^+\in \{ {\rm Re}z=\pi\}$.
Inequality \eqref{eq53} shows that the part of  the curve $\Gamma^+$
that is contained in $\{z\in \mathbb{C}: \pi-\frac{2}{5}< {\rm Re}z<\pi\}$, is contained in the set $\{z\in\mathbb{C}: \,{\rm Im}(z-\tilde{\eta} _M^+)>
{\rm Re} (z-\pi)\}$. A similar fact holds for  the curve $\Gamma_M^-$.
Hence, the set
\begin{align*}
\omega^{-1}\circ T_M&\Big(\big\{z\in \mathbb{C}: \,\;\;\;0<{\rm Im} z\;<\frac{2}{5} ,\,|{\rm Re} z|<M-\frac{1}{2}\big\}\Big)\nonumber\\=
(\mathfrak{I}\circ  T_M^{-1}\circ\omega)^{-1}&\Big(\big\{z\in \mathbb{C}: \,  -  \frac{2}{5}<    {\rm Re}z<\;0   ,\,|{\rm Im} z|<M-\frac{1}{2}\big\}\Big)
\end{align*}
contains the set \eqref{eqbr.28}.

By equation \eqref{eqbr.30} we obtain
\begin{align}\label{eqbr.25}
(\omega^{-1}\circ T_M)'(\pm (M-\frac{1}{2}))=i\frac{2\sqrt{M(M+1)}}{-(M-\frac{1}{2})^2+M(M+1)}= i\frac{\sqrt{M(M+1)}}{M-\frac{1}{8}}
\end{align}
Since $(T_M^{-1}\circ\omega)'(\tilde\eta_M^{\pm})=\frac{1}{(\omega^{-1}\circ T_M)'(\pm(M-\frac{1}{2}))}  $
we obtain
\begin{align}\label{eqbr.25a}
(T_M^{-1}\circ\omega)'(\tilde\eta_M^{\pm})=\frac{- i(M-\frac{1}{8})}{\sqrt{M(M+1)}}\,.
\end{align}
The derivative of the composition $\mathcal{G}_M^{(2)}\stackrel{def}=\mathfrak{g}_M\circ(iT_M^{-1})\circ \omega$ at the points $\tilde\eta^{\pm}$ equals
\begin{align}\label{eqbr.26}
(\mathcal{G}_M^{(2)})'(\tilde\eta_M^{\pm})=& \mathfrak{g}_M'(\pm i(M-\frac{1}{2}))\cdot ( (i T_M^{-1})\circ \omega)'(\tilde\eta_M^{\pm})\,.
\end{align}
By equation \eqref{eqbr.25a}
\begin{align}\label{eqbr26a}
(\mathcal{G}_M^{(2)})'(\tilde\eta_M^{-})= &- \frac{ \pi }{2} \frac{M-\frac{1}{8}}{\sqrt{M(M+1)}}i\,,\nonumber\\
(\mathcal{G}_M^{(2)})'(\tilde\eta_M^{+})= &- \frac{ \pi }{2} \frac{M-\frac{1}{8}}{\sqrt{M(M+1)}}i\,, \; \mbox{if}\; d=2M-1 \;\mbox{is even}\,, \nonumber\\
(\mathcal{G}_M^{(2)})'(\tilde\eta_M^{+})= &\;\;\;\;\frac{ \pi }{2} \frac{M-\frac{1}{8}}{\sqrt{M(M+1)}}i\,,\; \mbox{if}\; d \;\mbox{is odd}\,.
\end{align}
Note that  $\frac{M- \frac{1}{8}}{M+1} <\frac{M-\frac{1}{8}}{\sqrt{M(M+1)}}<1$. Further, the mapping $M\to\frac{M- \frac{1}{8}}{M+1}$   is increasing for $M$ increasing, hence $\frac{M-\frac{1}{8}}{\sqrt{M(M+1)}}>\frac{\frac{3}{2}-\frac{1}{8}}{\frac{5}{2}}>\frac{1}{2}$. It follows that
\begin{align}\label{eq3braids13}
\frac{\pi}{4 }<|(\mathcal{G}_M^{(2)})'(\tilde\eta_M^{\pm})|<\frac{\pi}{2}\,.
\end{align}
Put
\begin{align}\label{eq3braids11}
r_M\stackrel{def}=|(\mathcal{G}_M^{(2)})'(\tilde\eta_M^{\pm})|\,.
\end{align}
We define
\begin{align}\label{eq3braids7}
g_{M,2}(\zeta)=&
(\mathcal{G}_M^{(2)})(\frac{\zeta}{r_M})
\,, \; \zeta \in r_M{R}'_{M,2}\,,\nonumber\\
\eta_M^{\pm}\;\;\;=&r_M
\cdot \tilde\eta^{\pm}_M\,.
\end{align}

We will prove now that $r_M R'_{M,2}$ contains a disc of radius $\frac{1}{8}$ around $\eta_M^{\pm}$ and prove the required estimate of the second derive of $g_{M,2}$ in a disc of radius $\frac{1}{16}$ around $\eta_M^{\pm}$.
This will be done similarly as for elementary words of form (1).
Consider the half-strip $\mathfrak{S}_M^-\stackrel{def}=\{z\in\mathbb{C}: {\rm Re}z<\frac{\log 2}{\pi},\, |{\rm Im}z+M-\frac{1}{2}|<\frac{1}{2}\}\subset \{z\in\mathbb{C}: {\rm Re}z<\frac{\log 2}{\pi},\, |{\rm Im}z|<M\}$.
The restriction $\mathfrak{g}_M\mid \mathfrak{S}_M^-$
is a conformal mapping onto its image.
The image is the set
\begin{align*}
\Big\{\frac{1}{2}\cdot\exp\big(\pi(-i(M-\frac{1}{2})+\zeta)+\pi i(M-1)\big):\, \mbox{Re}\zeta<\frac{\log 2}{\pi}\,, |\mbox{Im}\zeta|<\frac{1}{2}&\Big\}\\
=\Big\{\frac{1}{2}\exp\big(-\frac{\pi}{2}i+\pi\zeta\big): {\rm Re}\zeta<\frac{\log 2}{\pi}   , |{\rm Im}\zeta|<\frac{1}{2}&\Big\}\\
=\Big\{re^{i\varphi}:0<r<1, -\pi<\varphi<0 &\Big\}\,.
\end{align*}
This set contains
the disc $\{|\zeta+\frac{i}{2}|<\frac{1}{2}\}$. Put $V_M^-=(\mathfrak{g}_M\mid \mathfrak{S}_M^-)^{-1}(\{|\zeta+\frac{i}{2}|<\frac{1}{2}\})$.
Let $\tilde{Q}_M^-=(\mathfrak{I}\circ T_M^{-1}\circ\omega)^{-1}(V_M^-)$.
We have the following sequence of conformal mappings
\begin{align}\label{eq3braids9}
\tilde{\eta}^-_M \; \overset{\mathfrak{I}\circ T_M^{-1}\circ\omega}{-\!\!\!-\!\!\!-\!\!\!-\!\!\!-\!\!\!-\!\!\!
\longrightarrow}
\; &-(M-\frac{1}{2})i \;\overset{\mathfrak{g}_M} {-\!\!\!-\!\!\!-\!\!\!-\!\!\!\longrightarrow} \;\;\; \;\;\; \;\;- \frac{1}{2}i\,\nonumber\\
\begin{sideways}{\begin{sideways}{\begin{sideways}$\in$\end{sideways}}\end{sideways}}\end{sideways}\quad \quad \quad \quad \quad  &\qquad \quad  \begin{sideways}{\begin{sideways}{\begin{sideways}$\in$\end{sideways}}\end{sideways}}\end{sideways}
\qquad \quad \quad\quad\quad \quad\quad\quad
\begin{sideways}{\begin{sideways}{\begin{sideways}$\in$\end{sideways}}\end{sideways}}\end{sideways}
\nonumber\\
\tilde{Q}_M^- \overset{\mathfrak{I}\circ T_M^{-1}\circ\omega}{-\!\!\!-\!\!\!-\!\!\!-\!\!\!-\!\!\!-\!\!\!\longrightarrow}& \;\;\; \;\;\; \;\;\;V_M^-\;\;\; \;\;\;\overset{\mathfrak{g}_M}{-\!\!\!-\!\!\!-\!\!\!-\!\!\!\longrightarrow} \{|\zeta+\frac{i}{2}|<\frac{1}{2}\}\;\; .\quad
\end{align}
The composition $\mathcal{G}_M^{(2)}=\mathfrak{g}_M\circ\mathfrak{I}\circ T_M^{-1}\circ\omega$ maps $\tilde{Q}_M^-$ conformally onto $\{|\zeta+\frac{i}{2}|<\frac{1}{2}\}$. The mapping $g_{M,2},\, g_{M,2}(\zeta)=\mathcal{G}_M^{(2)}(\frac{\zeta}{r_M})$
takes ${Q}_M^-\stackrel{def}=r_M\tilde{Q}_M^-$
conformally onto $\{|\zeta+\frac{i}{2}|<\frac{1}{2}\}$.  Moreover, the derivative of $g_{M,2}$ at
$\eta^{-}$ has absolute value equal to one (see equality \eqref{eq3braids11}).
Consider the inverse of the restriction of
$g_{M,2}$ to $Q_M^-$.
In the same way as for syllables of form (1) we see that $Q_M^-\subset r_M R'_{M,2}$ contains a disc of radius $\frac{1}{8}$ around $\eta_M^{-}$, and the absolute value of the second derivative of $g_{M,2}$ in a disc of radius $\frac{1}{16}$ around $\eta_M^{-}$ does not exceed $2^8$. The same arguments apply to $\eta_M^+$ instead of $\eta_M^-$.

The set $r_M R'_{M,2}$
also contains the set $r_M\widehat{R}_M$, and by \eqref{eqbr.28}
and inequality \eqref{eq3braids13} the set $r_M\widehat{R}_M$
contains the set
\begin{align}\label{eq3braids10}
\{z\in \mathbb{C}:- \frac{2}{5} \cdot\frac{\pi}{4}<{\rm Re}(z-\pi r_M)< 0,\;& {\rm Im}z-{\eta} _M^+<{\rm Re}( z-\pi r_M),\nonumber\\
&{\rm Im}z-{\eta} _M^->-{\rm Re}(z-\pi r_M)\}\,.
\end{align}
Since $\frac{1}{16}=0.0625<0.1314<\frac{\pi}{10}$, the union of the discs of radius $\frac{1}{16}$ around $\eta_M^{\pm}$ with the set \eqref{eq3braids10}  contains the rectangle
of horizontal side length equal to $\frac{1}{16\sqrt{2}}$ whose right vertices are  $\eta_M^+$ and $\eta_M^-$, respectively. For the syllable $_{pb}w_{pb}=_{pb}(a_2^{-1} a_1^{-1}\ldots)_{pb}$ of degree $d$ and
$M=\frac{d+1}{2}  $ we let $R_w$ be equal to this rectangle.  Further, we put
$\xi_w^{\pm}=\eta_M^{\pm}=r_M\tilde{\eta}_M^{\pm}$.  The points $\xi_w^{\pm}$ are the right vertices of the rectangle $R_w$. Hence, $R_w\subset r_M \mathring{R}_M$.
By  inequalities   \eqref{eq3braids.5} and    \eqref{eq3braids13}, and by equality \eqref{eq3braids11} the inequalities
\begin{align*}
 \mbox{vsl}(R_w) \leq r_M \mbox{vsl}( \mathring{R}_M)
\leq \frac{\pi}{2}\cdot 2 \log(2d+\sqrt{(4d^2-1)})= \pi
 \mathcal{L}(w)\,.
\end{align*}
hold.

Put $g_w=g_{M,2}$. The conditions of the lemma concerning the first derivatives $g'_w(\xi_w^{\pm})$ are satisfied.
\index{$R_{M,2}$} \index{$g_{M,2}$} \index{$\eta_M^{\pm}$}

Lemma \ref{lemxxx}  is proved for syllables of the form $_{pb}w_{pb}=_{pb}(a_2^{-1}a_1^{-1}\ldots)_{pb}$ of degree $d\geq 2$. For the other elementary words of form $(2)$ with $pb$ horizontal boundary values and degree $d\geq 2$ the
Lemma is obtained by symmetry arguments. The points $\xi_w^{\pm}$ are either the right vertices or the left vertices of the rectangle $R_w$.
\hfill $\Box$

\medskip

\noindent {\bf Proof of the upper bound in Theorems  \ref{thm1}, \ref{thm10.1'}, and \ref{thm2}.}
Take any word $w$ with $pb$ horizontal boundary values. Write it as product
of syllables $_{pb}(w_j)_{pb}$ with $pb$ horizontal boundary values.
By Lemma \ref{lemxxx} each $_{pb}(w_j)_{pb}$ can be represented by a holomorphic mapping $g_{w_j}$ from the rectangle $R_{w_j}$ into $\mathbb{C}\setminus\{-1,1\}$. Recall that the vertical side length of  $R_{w_j}$ does not exceed $\pi \, \mathcal{L} (w_j)$, the horizontal side length of $R_{w_j}$ is at least $\frac{1}{16\sqrt{2}}$. We shrink the rectangles $R_{w_j}$ in the horizontal direction, so that its horizontal side length is exactly equal to $\frac{1}{16\sqrt{2}}$, and the points  $\xi_{w_j}^{\pm}$ are on the boundary of the shrinked rectangle $R_{w_j}$. Moreover, they are midpoints of the open horizontal sides of the shrinked rectangles, if they were midpoints of the open horizontal sides of the original rectangle, and they are on the open right (left) side of the shrinked rectangle, if they were on the open right (left) side of the original rectangle.

We may assume that  $g_{w_j}$ extends holomorphically to discs of radius $\frac{1}{16}$ around the points  $\xi_{w_j}^{\pm}$.

Take complex numbers $a_j, , j=1,\ldots, d,$ such that for the translated rectangles $R_j\stackrel{def}=R_{w_j}+a_j$ and the translated points $\xi_{j}^{\pm}\stackrel{def}=\xi_{w_j}^{\pm}+a_j$  the equalities $\xi_j^+=\xi_{j+1}^-$ hold for $j\leq N-1$. Choose complex numbers $b_j$ so that for the translated functions  $g_j(\zeta)\stackrel{def}=b_j+ g_{w_j}(\zeta-a_j) $
the equalities $g_j(\xi_j^+)=g_{j+1}(\xi_{j+1}^-),\, j=1,\ldots,N-1,$ hold.

We will perform quasiconformal gluing of the $g_j$ using the fact that each $g_j$ extends holomorphically to a disc of radius $\frac{1}{16}$ around $\xi_j^{\pm}$.
 Denote the extended functions again by $g_j$. The disc of radius  $\frac{1}{16}$
around  $\xi_j^{+}$ contains the square $Q_j$ of side length $\frac{\sqrt{2}}{16}$ with center $\xi_j^{+}$.
We define a smooth mapping $g$ on
\begin{align*}
\mathfrak{R}'_w\stackrel{def}= \mbox{Int}\left(\,\overline{\bigcup_{j=1}^N {R_j}}\,\right)\cup \bigcup _{j=1}^{N-1} Q_j
\end{align*}
as follows.

Consider the $C^1$-function $\chi_0$ on the interval $[0,1]$, $\chi_0(t)=6 \int_0^t \tau (1-\tau) d\tau$. Then $\chi_0(0)=0$, $\chi_0(1)=1$, and $0\leq \chi_0'(t)\leq 6t(1-t)\leq \frac{3}{2}$.
Let $a<2^{-5}$ be a positive number that will be chosen in a moment.

If $\zeta\in \mathfrak{R}'_w$ and $|\mbox{Im}\zeta-\xi_j^+|\leq \frac{a}{2}$ for some $j=1,\dots,N-1,$  then $\zeta \in Q_j$, and
$g_j$ and $g_{j+1}$ are defined and holomorphic in $Q_j$.
In this case we put on $Q_j$
\begin{align*}
g(\zeta)=(1-\chi_j)(\mbox{Im}\zeta)\cdot g_j(\zeta)+\chi_{j}(\mbox{Im}\zeta)\cdot g_{j+1}(\zeta)\,.
\end{align*}
Here $\chi_j(t)=
\chi_0(\frac{t-\mbox{Im}\xi_j^+  +\frac{a}{2}} {a})$. If $|{\rm Im}\zeta-\xi_j^{\pm}|>\frac{a}{2}$ for all $j=1,\ldots,N-1,$ then $\zeta$ is contained in  a single rectangle $R_j$, and we put $g=g_j$. The mapping $g$ is of class $C^1$. If $|{\rm Im}\zeta-\xi_j^+|\leq \frac{a}{2}$ for some $j=1,\dots,N-1,$ then for the Beltrami coefficient we obtain with $\zeta=u+iv$
$$
\mu_g(\zeta)= \frac{\frac{\partial}{\partial \overline \zeta}g(\zeta)}{\frac{\partial}{\partial \zeta}g(\zeta)} = \frac{\frac{i}{2}( \frac{\partial}{\partial v}{\chi_j}\cdot (g_{j+1} -g_{j}))(\zeta)}{(\frac{-i}{2} \frac{\partial}{\partial v}{\chi_j}\cdot (g_{j+1} -g_{j})+ (1-\chi_j)\cdot g_j' + \chi_j\cdot g_{j+1}')(\zeta)}.
$$

At all other points the Beltrami coefficient equals zero.
For $\zeta$ in the square of side length  $a$ with center $\xi^+_j$ the inequalities $|g'_j(\zeta)-g'_{j}(\xi_j^+)|\leq  2^8 |\zeta-\xi_j^+|$ and  $|g'_{j+1}(\zeta)-g'_{j+1}(\xi_j^+)|\leq  2^8 |\zeta-\xi_j^+|$ hold. Since $g_j(\xi_j^+)=g_{j+1}(\xi_j^+)$,
and $g'_j(\xi_j^+)=g'_{j+1}(\xi_j^+)$,
we obtain on this square
$$
|g_j-g_{j+1}|\leq 2\cdot\frac{1}{2}\cdot 2^8 \cdot\frac{a^2}{2}=2^7 a^2,
$$
and
$$
 (1-\chi_j)\cdot| g_j' - g'_j(\xi_j^+)|+ \chi_j\cdot |g_{j+1}'-(g_{j+1}')(\xi_j^+)|\leq 2^7\sqrt{2}a\,.
$$
Since  $g'_j(\xi_j^+)=g'_{j+1}(\xi_j^+)$ has absolute value equal to $1$ and $|\chi_j'|\leq \frac{3}{2}\cdot \frac{1}{a}$, we obtain
\begin{align*}
|\mu_g(\xi)|\leq \frac{\frac{1}{2}\cdot\frac{3}{2}\cdot 2^7 a}{1-\frac{1}{2}\cdot\frac{3}{2}\cdot 2^7 a-2^7\sqrt{2} a}=\frac{\frac{3}{4}2^7a} {1 -(\frac{3}{4}+\sqrt{2})2^7a}\,.
\end {align*}
Put $a=\frac{1}{3}2^{-8}$.  Since $\frac{3}{4}+\sqrt{2}<3$,we obtain  the inequality
\begin{align*}
|\mu_g(\zeta)|\leq \frac{\frac{3}{4}\cdot \frac{1}{6}}{1- (\frac{3}{4}+\sqrt{2}) \cdot\frac{1}{6}}  < \frac{1}{4}   \,.
\end{align*}

The quasiconformal dilatation $K$ is less than $\frac{5}{3}$.
The set $\mathfrak{R}'_w$ contains a curvilinear rectangle $\mathfrak{R}_w$ of the form $R_{J,\Phi, {\sf b}}$ (see Example 3, Section \ref{sec:4.1a}) with $|J|\leq \pi\, \mathcal{L}(w)$, ${\sf b}=a$, and a smooth function $\Phi$. For each positive $\varepsilon$ the function $\Phi$ can be chosen with $|\Phi'|\leq 1+\varepsilon$. Indeed, replace first each $R_j$ by a rectangle $R_j^a$ contained in $R_j$ of horizontal side length $a$ and vertical side length ${\rm vsl}(R_j^a)={\rm Im}(\xi_j^+-\xi_j^-)$, so that the $\xi_j^{\pm}$ are the midpoints of the horizontal sides of the new rectangle $R_j^a$ if they are midpoints of the horizontal sides of $R_j$, and they are equal to the endpoints of the right (left) sides of the $R_j^a$, if they are contained in the open right (left) sides of the $R_j$.
For $j\leq N-1$ we consider the intersection of $\overline{R^a_j}\cup \overline{R^a_{j+1}}$ with the strip $\{|{\rm Im}\zeta-{\rm Im}\xi_j^+|\leq \frac{a}{2}\}$. Replace the intersection by the parallelogram whose horizontal sides are equal to $R_j^a \cap \{{\rm Im}\zeta
={\rm Im}\xi_j^+ -\frac{a}{2}\}$ and $R_{j+1}^a \cap \{ {\rm Im}\zeta
={\rm Im}\xi_j^+ +\frac{a}{2}\}$, respectively. Doing so for all $j\leq N-1$  we obtain a curvilinear rectangle of the form   $R_{J,\Phi_0, {\sf b}}$ for a piecewise $C^1$ function $\Phi_0$ with $|\Phi_0'|\leq 1$.  For any positive $\varepsilon$ we obtain after suitable smoothing a curvilinear rectangle of the required form with a function $\Phi$ that satisfies $|\Phi'|\leq 1+\varepsilon$. See Figure \ref{Fig3-braids11.pdf}.

\begin{figure}[H]
\begin{center}
\includegraphics[width=5cm]{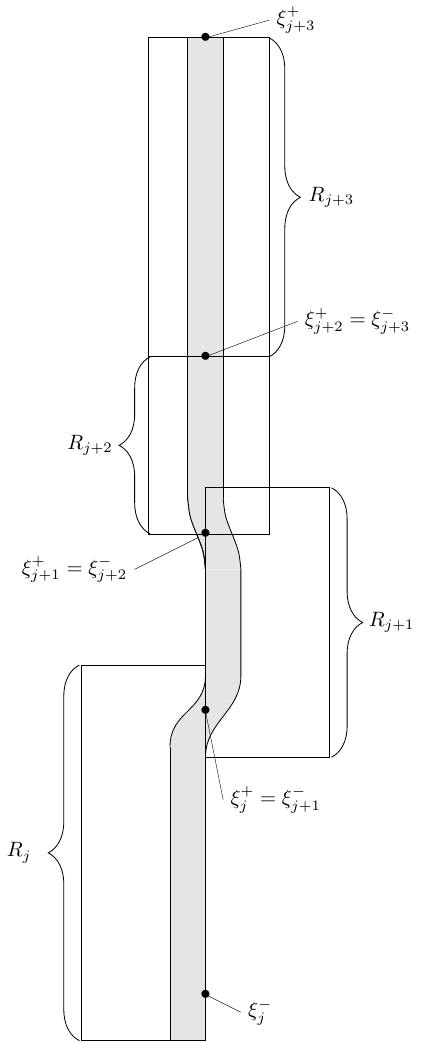}
\end{center}
\caption{A curvilinear rectangle $\mathfrak{R}_w$ that admits a holomorphic mapping
representing $_{pb}w_{pb}$}  \label{Fig3-braids11.pdf}
\end{figure}

Choose $\varepsilon$ and consider the obtained curvilinear rectangle  $\mathfrak{R}_w$. The restriction $g\mid \mathfrak{R}_w$ represents $_{pb}w_{pb}$. Since $\mbox{vsl}(R_w)< \pi \mathcal{L}(w)$, the extremal length $\lambda(\mathfrak{R}_w)$ is less than $(1+(1+\varepsilon)^2) \cdot \frac{\pi \mathcal{L}(w)}{a}$ (see Lemma \ref{lemm216}).
Using the normalized solution $\omega$ of the Beltrami equation on the complex plane with Beltrami coefficient $\mu_g$ on $\mathfrak{R}_w  $ and $0$ else,  and the conformal mapping $\psi$ from a true rectangle $\mathcal{R}_w$ onto the curvilinear rectangle $\omega(\mathfrak{R}_w)$,  we obtain the holomorphic mapping $g\circ\omega^{-1}\circ \psi:\mathcal{R}_w \to \mathbb{C}\setminus\{-1,1\}$ representing  $_{pb}w_{pb}$. The extremal length of $\mathcal{R}_w$ is less than
\begin{align*}
&K \cdot (1+(1+\varepsilon)^2) \cdot \frac{\pi \mathcal{L}(w)}{a}<
\frac{5}{3}\cdot  (1+(1+\varepsilon)^2) \cdot  \frac{ \pi}{a} \cdot  \mathcal{L}(w)\nonumber\\
= & (1+(1+\varepsilon)^2) \cdot 5\pi\cdot 2^{8}  \cdot\mathcal{L}(w)< 2^{13} (1+(1+\varepsilon)^2) \mathcal{L}(w) \,.
\end{align*}
We used the inequality $5\pi<2^4$.
Choosing $\varepsilon$ small enough, we get a curvilinear rectangle $\mathcal{R}_w$ with $\Lambda(\mathcal{R}_w)<  2^{13} \cdot \mathcal{L}(w) $, that admits a holomorphic mapping representing $w$. This proves Theorem \ref{thm1}  in case of $pb$ boundary values.

In the case, when $w$ has at least two syllables and totally real right boundary values, or totally real left boundary values, or both boundary values are totally real, the first or last syllable, or both have mixed boundary values. Suppose, the respective syllable $\mathfrak{s}$ is of form $(1)$ or $(3)$.
A lift of the syllable  $\mathfrak{s}$ with the required mixed boundary values can be represented by a holomorphic mapping from a rectangle with the same estimates for the vertical side length and the horizontal side length as in the case of $pb$ boundary values. Moreover, the value of the derivative of the mapping at the point, in a neighbourhood of which the gluing is performed, is the same as for $pb$ boundary values, and the estimate for the second derivative is the same.

Suppose the respective syllable $\mathfrak{s}$ is of form $(2)$. In this case instead of representing the lift of  $\mathfrak{s}$ with mixed boundary values, we represent the lifts under $f_1\circ f_2$ of a singleton $\mathfrak{s}'$ with mixed boundary values and the lift of a word $\mathfrak{s}''$ with one letter less than  $\mathfrak{s}$ with $pb$ horizontal boundary values.
Let $g_{\mathfrak{s'}}:R _{\mathfrak{s'}}\to \mathbb{C}\setminus i\mathbb{Z},$ and  $g_{\mathfrak{s''}}:R _{\mathfrak{s''}}\to \mathbb{C}\setminus i\mathbb{Z},$ be the respective representing mappings. We may choose the mappings so that the derivative at the points  in a neighbourhood of which the gluing is performed equals $\pm i$, the estimate for the horizontal side length of each rectangle and the estimate of the second derivative of the mappings is as in the case of $pb$ boundary values, and
${\rm vsl}(R _{\mathfrak{s'}}) \leq \pi \mathcal{L}(\mathfrak{s'})\leq \pi \mathcal{L}(\mathfrak{s})$ and
${\rm vsl}(R _{\mathfrak{s''}}) \leq \pi \mathcal{L}(\mathfrak{s''})\leq \pi \mathcal{L}(\mathfrak{s})$.
The quasiconformal gluing and the choice of the curvilinear rectangle $\mathcal{R}$ is done as in the case of $pb$ boundary values. The extremal length of the curvilinear rectangle $\mathcal{R}$ differs no more than by a factor $2$ from the estimate in the case of $pb$ boundary values.

If $w$ consists of a single syllable the upper bound in the case of mixed boundary values
(Theorem \ref{thm10.1'}) follows directly from Proposition \ref{prop4b}. The exceptional cases were treated in the proof of the lower bound in Theorem \ref{thm1}.
Theorem \ref{thm1}  and Theorem  \ref{thm10.1'} are proved.

Theorem \ref{thm2} in the non-exceptional cases is obtained in the same way as Theorem  \ref{thm1}  in the case of $pb$ horizontal boundary values. The only difference is that we have to perform also quasiconformal gluing of
the last syllable to the first syllable of the word.
This proves Theorem  \ref{thm2}. \hfill $\Box$

\section{The extremal length of $3$-braids.}\label{sec:3-braids5}
In this section we come back to braids. Recall that we may consider $n$-braids as elements of the fundamental
group of $C_n(\mathbb{C})\diagup\mathcal{S}_n$.

In Section \ref{sec:4.1b} we saw that the totally real subspace $C_n (\mathbb {R}) \diagup {\mathcal S}_n$ of the $n$-dimensional symmetrized configuration space is connected and simply connected, and
the fundamental group $\pi_1(C_n (\mathbb {C}) \diagup {\mathcal S}_n, E_n)$ of the symmetrized configuration space with base point $E_n\in C_n (\mathbb {R}) \diagup {\mathcal S}_n$ is canonically isomorphic to the relative fundamental group
$\pi_1(C_n (\mathbb {C}) \diagup {\mathcal S}_n, C_n (\mathbb {R}) \diagup {\mathcal S}_n   )$.

We recall the definition of the extremal length of braids with totally real horizontal boundary values that was given in Section \ref{sec:4.1b}, Definition \ref{def1}.
For a braid $b \in \mathcal{B}_n$ the value $\Lambda_{tr}(b)=\Lambda(b_{tr})$ is the infimum of the extremal length of rectangles that admit a
holomorphic mapping into $C_n(\mathbb{C})\diagup \mathcal{S}_n$ that represents the image $b_{tr}$ of $b$ under the canonical isomorphism
$\pi_1(C_n(\mathbb {C}) \diagup {\mathcal S}_n, E_n)\to \pi_1(C_n(\mathbb {C}) \diagup {\mathcal S}_n,  C_n(\mathbb {R}) \diagup {\mathcal S}_n    )$.

We
 will consider now the case $n=3$.
By Lemma \ref{lemm1a} the group $\mathcal{PB}_3\diagup {\langle\Delta_3^2 \rangle}$ is isomorphic to the fundamental group of $\mathbb{C}\setminus \{-1,1\}$. Hence for pure braids $b\in \mathcal{PB}$ the estimate of the extremal length $\Lambda(b_{tr})$
with totally real boundary values can be given in terms of the class of $b$ in  $\mathcal{PB}_3\diagup {\langle\Delta_3^2 \rangle}$ which can be identified with an element of the fundamental group of the twice punctured complex plane (see Theorem \ref{thm1}).

We consider now arbitrary $3$-braids (not necessarily pure braids).
The following lemma is needed to reduce the case of arbitrary $3$-braids to the case of pure $3$-braids.
\begin{lemm}\label{lemm1'}
Any braid $b\in \mathcal{B}_3$ which is not a power of $\Delta_3$ can be written in a unique way in the form
\begin{equation}\label{eq2'}
\sigma_j^k \, b_1 \, \Delta_3^{\ell}\,
\end{equation}
where $j=1$ or $j=2$, $k\neq 0$ is an integer, $\ell$ is a (not necessarily even) integer, and $b_1$ is a word in $\sigma_1^2$ and $\sigma_2^2$ in reduced form. If
$b_1$ is not the identity, then the first term of $b_1$ is a non-zero even
power of $\sigma_2$ if $j=1$, and  $b_1$ is a non-zero even  power of
$\sigma_1$ if $j=2$.
\end{lemm}

For an integer $j\neq 0$ we denote by $q(j)$ that even integer neighbour
of $j$, which is closest to zero. In other words, $q(j)=j$
for each even integer $j\neq 0$. For each odd integer $j\,,$  $q(j)= j
-\mbox{sgn}(j)$, where  $\mbox{sgn}(j)$
for a non-zero integer number $j$ equals $1$ if $j$ is positive,
and $-1$ if $j$ is negative. For a braid in form \eqref{eq2'} we put
$\vartheta(b) \stackrel{def}{=}\sigma_j^{q(k)} \, b_1$. \index{$\vartheta(b)$}
\begin{thm}\label{thmbr.3}
Let $b \in \mathcal{B}_3$ be a (not necessarily pure) braid which is not a power of $\Delta_3$, and let $w$ be the word representing the image of $\vartheta(b)$ in $\mathcal{B}_3 \diagup \langle \Delta_3^2\rangle$.
Then
$$
\frac{1}{2\pi}\mathcal{L}(w) \leq \Lambda_{tr}(b)  \leq 2^{14} \cdot
\mathcal{L}(w)  \,,
$$
except in the case when $b=\sigma_j^{k}\,\Delta_3^{\ell}$ where $j=1$ or  $j=2$, $k\neq 0$ is an integer number, and $\ell$ is an arbitrary integer. In the exceptional case $\Lambda_{tr}(b)=0$.
\end{thm}
Notice that the periodic $3$-braids are among the exceptional cases. We have $\vartheta(\Delta_3^{\ell})={\rm id}$, and  $\vartheta(\sigma_1\sigma_2\Delta_3^{\ell})= \vartheta(\sigma_2^{-1}\Delta_3^{\ell+1})= id$, and similarly for $(\sigma_1\sigma_2)^{-1}\Delta_3^{\ell}$ and  $(\sigma_2\sigma_1)^{\pm 1} \Delta_3^{\ell}$ .

\smallskip

\noindent {\bf Proof of Lemma \ref{lemm1'}.}
Let $\tau_3$ be the natural homomorphism from the braid group $\mathcal{B}_3$ to the symmetric group $\mathcal{S}_3$. When $\tau_3(b)$ is the identity then the braid $b$ is pure and the statement is clear.

Assume $\tau_3(b)$ is not the identity. We consider first the case when $\tau_3(b)$ is a transposition. If $\tau_3(b)=(13)$ then $b \Delta_3^{-1}$ is a pure braid, hence if $b$ is not a power of $\Delta_3$ it can written in the form \eqref{eq2'} with $k$ even and $\ell$ odd.
If $\tau_3(b)=(12)$ then $\sigma_1^{-1}b$ is a pure braid. Hence,  $b$
can be written in the form \eqref{eq2'} with $j=1$, $k$ odd and $\ell$ even. If $\tau_3(b)=(23)$ then $\sigma_2^{-1}b$ is a pure braid. Hence, $b$ can be written in the form \eqref{eq2'} with $j=2$, $k$ odd and $\ell$ even.

It remains to consider the case when $\tau_3(b)$ is a cycle. Suppose, $\tau_3(b)=(123)$. Since $(123)(13)=(12)$ the relation $\tau_3(b)(13)^{-1}=(12)$ holds. Hence, $\sigma_1^{-1}b \Delta_3^{-1}$ is a pure braid and $b$ has the form \eqref{eq2'} with $j=1$, and $k$ and $\ell$ odd. If $\tau_3(b)=(132)$ then, since $(132)(13)=(23)$ the braid  $\sigma_2^{-1}b \Delta_3^{-1}$ is a pure braid and $b$ can be written in the form \eqref{eq2'} with $j=2$, and $k$ and $\ell$ odd. \hfill $\Box$

\medskip

\noindent {\bf Proof of the lower bound in Theorem \ref{thmbr.3}.}
Since $\Lambda(b_{tr})=\Lambda((b \Delta_3)_{tr})$ (see Lemma \ref{lemm1}) we may suppose that
$b=\sigma_j^{k} b_1$ for a pure braid $b_1$ which is a reduced word in $\sigma_1^2$ and $\sigma_2^2$. We may suppose that $k$ is an odd integer. The case when $k$ is even follows from Theorem \ref{thm1} and Lemma \ref{lemm1a}. Let $j=1$.
Consider a holomorphic mapping $g$ of a rectangle $R$ to $C_3(\mathbb{C)} \diagup \mathcal{S}_3$ that represents the braid and lift it to a mapping $\tilde g$ into $C_3(\mathbb{C)}$ so that the continuous extension of the mapping to the closed rectangle takes the open lower side to  $\{(x_1,x_2,x_3) \in
\mathbb{R}^3: x_2<x_1<x_3\}$.
Apply the mapping $\mathfrak{C}$. (For notation see also the proof of Lemma \ref{lemm1a}). We obtain a holomorphic mapping $\mathfrak{C}(\tilde g):R\to \mathbb{C}\setminus \{-1,1\}$,   $\mathfrak{C}(\tilde{g})(z) = 2
\,\frac{\tilde g_2(z)-\tilde g_1(z)}{\tilde g_3(z)-\tilde g_1(z)}\,-1,\, z\in R$,
whose continuous extension to the closure takes the open lower side of the rectangle to $(-\infty,-1)$.

Consider first the case when $|k|\geq 3$. We represent $b_{tr}$ by a curve $\gamma$ in $C_3(\mathbb{C})\diagup \mathcal{S}_3$, such that taking the lift $\tilde{\gamma}$ with initial point in $\{(x_1,x_2,x_3) \in \mathbb{R}^3: x_2<x_1<x_3\}$, the curve $\mathfrak{C}(\tilde{\gamma})$
is the union of the following two curves: a half-circle in $\mathbb{C}\setminus \{-1,1\}$ with endpoints in $\mathbb{R}\setminus \{-1,1\}$ that is contained in the lower half-plane if $k>0$, or in the upper half-plane if $k<0$, and joins a point in $(-\infty,-1)$ with a point in $(-1,1)$, and a curve representing the element $\mathfrak{C}_*(\vartheta(b)_{tr}) \in \pi_1^{tr}$. It can be proved along the same lines as the proof of Statement 3 of Lemma \ref{lemm15'}, that the rectangle contains a curvilinear rectangle such that the restriction of the mapping $\mathfrak{C}(\tilde g)$ to it represents $\mathfrak{C}_*(\vartheta(b)_{tr})$. Therefore $\Lambda(b_{tr}) \geq \Lambda(\vartheta(b)_{tr})$. If $\vartheta(b)\neq \sigma_1^{2k'} $ for an integer $k'$, in other words, if $\vartheta(b)$ is not among the exceptional cases of
Theorem \ref{thm1}, the lower bound holds.

Suppose $|k| =1$. If $b\neq \sigma_1^{\pm 1} $
then $b_{tr}$ can be represented by a curve $\gamma$ for whose lift $\tilde \gamma$ with initial point in $\{(x_1,x_2,x_3) \in \mathbb{R}^3: x_2<x_1<x_3\}$ the curve $\mathfrak{C}(\tilde{\gamma})$
is the union of the following two curves:
a quarter-circle $\Gamma_0$ in $\mathbb{C}\setminus \{-1,1\}$
in the upper or lower half-plane which joins a point in $(-\infty,-1)$ with a point in $i\mathbb{R}$, and a curve $\Gamma_1$ that represents $\mathfrak{C}_*(_{pb}\vartheta(b)_{tr})$.  Following along the lines of proof of Statement 3 of Lemma \ref{lemm15'}, we see that the rectangle contains a curvilinear rectangle such that the restriction of the mapping $\mathfrak{C}(\tilde g)$ to it represents $\mathfrak{C}_*(_{pb}\vartheta(b)_{tr})$. The lower bound
for this case follows from Theorem \ref{thm10.1'}.

The proof
for the case $b=\sigma_2^k b_1$ (with $b_1$ a reduced word in $\sigma_1^2$ and $\sigma_2^2$)  is similar and is left to the reader. In this case the lift to $C_3(\mathbb{C})$ of a representing mapping for $b_{tr}$ is chosen with initial point in $\{x_1<x_3<x_2\}$.

The lower bound of Theorem \ref{thmbr.3} is proved. \hfill $\Box$

\medskip

\noindent {\bf Proof of the upper bound of Theorem \ref{thmbr.3}}.
We prove the upper bound for $3$-braids that are not pure and not among the exceptional cases. We write again such a braid in the form \eqref{eq2'}, and assume that the braid equals $b=\sigma_1^{k} b_1$
for an odd natural number $k$ and $b_1$ a pure braid as in Lemma \ref{lemm1'} which is not the identity. We want to find a holomorphic mapping $g_2$  from a rectangle $\mathbb{C}\setminus\{-1,1\}$, such that the mapping
$(-1,g_2,1)$
from the rectangle into $C_3(\mathbb{C})$ projects under $\mathcal{P}_{\rm sym}$ to a holomorphic maping that represents  $b_{tr}$.
If $k\geq 3$ (or $k\leq -3$, respectively) we first represent a lift under $f_1\circ f_2$ of the homotopy class of a half-circle in the lower half-plane (in the upper half-plane, respectively),
that joins $(-\infty,-1)$ with $(-1,1)$, by a suitable holomorphic mapping $g$ from a rectangle to $\mathbb{C}\setminus iZ$. Then we represent
a lift of $\mathfrak{C}_*(_{tr}\vartheta(b)_{tr})$ under $f_1\circ f_2$ by a suitable holomorphic mapping from a rectangle to $\mathbb{C}\setminus iZ$. Quasiconformal gluing of the two mappings yields a holomorphic mapping from a rectangle to  $\mathbb{C}\setminus iZ$. Denote the projection under $f_1\circ f_2$ of this mapping by $g_2$. The mapping $(-1,g_2,1)$ defines a holomorphic
mapping of a rectangle into $C_3(\mathbb{C})$ whose projection under $\mathcal{P}_{\rm sym}$ represents  $b_{tr}$.

If $k=1$ we represent a lift under $f_1\circ f_2$ of the homotopy class of a quarter-circle
that joins $(-\infty,-1)$ with the imaginary axis by a suitable holomorphic mapping from a rectangle to $\mathbb{C}\setminus i\mathbb{Z}$,
and represent
a lift of $\mathfrak{C}_*(_{pb}\vartheta(b)_{tr})$ under $f_1\circ f_2$ by a suitable holomorphic mapping from a rectangle to $\mathbb{C}\setminus iZ$. After quasiconformal gluing of the two mappings we finish the proof of the upper bound as in the previous case. We leave the details to the reader.
\hfill $\Box$

\medskip

The space $C_3(\mathbb{R}\diagup \mathcal{S}_3)$ has a disadvantage especially when we want to deal with conjugacy classes of $3$-braids. Namely,
the codimension of   $C_3(\mathbb{R}\diagup \mathcal{S}_3)$ in  $C_3(\mathbb{C}\diagup \mathcal{S}_3)$   is bigger than one, hence, each curve in generic position avoids this space. It is convenient to consider the set   \index{$\mathcal{H}$}
\begin{align}\label{eq3-braids.4}
\mathcal{H} \stackrel{def}= & \{ \{z_1,z_2,z_3\} \in C_3(\mathbb{C})\diagup \mathcal{S}_3: \mbox{the three points}\; z_1,z_2, z_3  \nonumber\\
& \mbox{are contained in a real line in the complex plane}\}\,.
\end{align}
The set $\mathcal{H}$ is a smooth real hypersurface of $C_3(\mathbb{C})\diagup \mathcal{S}_3$. Indeed, let $\{z_1^0,z_2^0,z_3^0\}$ be a point of the symmetrized configuration space. Introduce coordinates near this point by lifting a neighbourhood of the point to a connected open set
in $C_3(\mathbb{C})$ equipped with coordinates $(z_1,z_2,z_3)$. Since the linear map  $M(z)\stackrel{def}= \frac{z-z_1}{z_3-z_1},\, z \in \mathbb{C},$ maps the points $z_1$ and $z_3$ to the real axis, the three points $z_1, \, z_2,$ and $z_3$  lie on a real line in the complex plane iff the imaginary part of $z'_2\stackrel{def}=M(z_2)= \frac{z_2-z_1}{z_3-z_1}$ vanishes. The equation $\mbox{Im}\frac{z_2-z_1}{z_3-z_1}=0$ in local coordinates $(z_1,z_2,z_3)$ defines a local piece of a smooth real hypersurface.

The set $\mathcal{H}$ is not simply connected. 
Nevertheless we may define homotopy classes of curves in $C_3(\mathbb{C})\diagup \mathcal{S}_3$ with endpoints in this set.

For later use we need the following lemmas.

\begin{lemm}\label{lemm3-braids1}
Let $\gamma$ be a loop in $\mathcal{H}$ with base point $\{-1,0,1\}\in \mathcal{H}$. Then  for some $k\in \mathbb{Z}$ the loop $\gamma$ represents the braid $\Delta_3^k$
with this base point.
\end{lemm}

\noindent {\bf Proof.}
Consider the lift  $\tilde{\gamma}=(\tilde{\gamma}_1, \tilde{\gamma}_2, \tilde{\gamma}_3):[0,1]\to C_3(\mathbb{C})$  of $\gamma$ with $\tilde{\gamma}(0)=(-1,0,1)$. By a homotopy of $\gamma$ in $ \mathcal{H}$ with fixed endpoints we may assume that for the lift  $\tilde{\gamma}=(\tilde{\gamma}_1, \tilde{\gamma}_2, \tilde{\gamma}_3)$ for all  $t\in [0,1]$ the equality $\tilde{\gamma}_2(t)=0$ holds and 
the segment $[\tilde{\gamma}_1(t),\tilde{\gamma}_3(t)]$ has length $2$. Then
for each point $\tilde{\gamma}(t)$ there is a unique complex number $e^{i\alpha(t)}$ 
such that $\alpha(0)=0$ and 
$ \tilde{\gamma}(t)=e^{i\alpha(t)}\big((-1,0,1)\big)$ 
and the $\alpha(t)$ 
depend continuously on $t\in [0,1]$. Since $\tilde{\gamma}(1)$ is equal to $(-1,0,1)$ or to $(1,0,-1)$, the equality $e^{i\alpha(1)} = \pm 1$ holds. 
The curve $t\to e^{i\alpha(t)},\,t\in[0,1],$ is homotopic with fixed endpoints as a curve in the unit circle 
to the curve $t\to e^{\pi k t i},\,t\in[0,1] $. Hence, $ \tilde{\gamma}$ is homotopic with fixed endpoints to the curve $t\to    e^{\pi k t i}(-1,0,1),\,t\in[0,1] $ for an integer number $k$.

After applying the projection $\mathcal{P}_{\rm sym}$ we obtain the curve
$t\to    e^{\pi k t i}\{-1,0,1\}$, $t\in[0,1] $,
in  $C_3(\mathbb{C})\diagup \mathcal{S}_3$ that is homotopic with fixed base point to $\gamma$
and represents $\Delta_3^k$.  \hfill $\Box$

\begin{lemm}\label{lemm3-braids2}
Suppose for a conjugacy class $\hat b$ there exists a representing loop that avoids $\mathcal{H}$. Then  $\hat b$ is the conjugacy class of a periodic braid of the form $(\sigma_1 \sigma_2)^k$ for an integer $k$.
\end{lemm}

\noindent {\bf Proof.} Let $\gamma$ be a loop that represents $\hat b$ and avoids
$\mathcal{H}$. Choose a base point on $\gamma$, so that $\gamma$ represents a braid $b$. For some natural number $k$ the power $b^k$ of $b$ is a pure braid, and the curve
 $\gamma^k\stackrel{def}=\underbrace{\gamma\circ\ldots \gamma}_{\substack{k}}$ represents the pure braid $b^k$ and does not meet $\mathcal{H}$.
Then $\mathfrak{C}({\gamma^k})=1-2\frac{\gamma^k_2-\gamma^k_1}{\gamma^k_3-\gamma^k_1}$
is a loop in $\mathbb{C}\setminus\{-1,1\}$ that avoids $\mathbb{R}$. Here $\gamma^k_j$ are the coordinate functions of $\gamma^k$. This means that $\mathfrak{C}({\gamma^k})$ is contractible, hence
$\gamma^k$ represents $\Delta^{2\ell}$ for an integer $\ell$, and hence $\gamma$ represents a periodic braid.

If a representative $\gamma:[0,1]\to C_3(\mathbb{C})\diagup \mathcal{S}_3,\, \gamma(0)=\gamma(1),$ of a $3$-braid $b$ avoids $\mathcal{H}$, then the associated permutation $\tau_3(b)$ cannot be a transposition. Indeed, assume the contrary. Then there is a lift $\tilde \gamma$ of $\gamma$ to $C_3(\mathbb{C})$, for which $(\tilde{ \gamma}_1(1),\tilde{ \gamma}_2(1),\tilde{ \gamma}_3(1))=(\tilde{ \gamma}_3(0),\tilde{ \gamma}_2(0),\tilde{\gamma}_1(0))$.
Let $L_t$ be the line in $\mathbb{C}$ that contains  $\tilde{\gamma}_1(t)$ and $\tilde{\gamma }_3(t)$, and is oriented so that running along $L_t$ in positive direction we meet first $\tilde{\gamma }_1(t)$ and then $\tilde{ \gamma}_3(t)$.
The point $\tilde\gamma_2(0)$ is not on $L_0$. Assume without loss of generality, that it is on the left of $L_0$ with the chosen orientation of $L_0$. Since for each $t\in [0,1]$ the three points $\tilde{\gamma }_1(t),\,\tilde{ \gamma}_2(t)$ and $\tilde{\gamma}_3(t)$ in $\mathbb{C}$ are not on a real line, the point $\tilde{\gamma}_2(t)$ is on the left of $L_t$ with the chosen orientation. But the unoriented lines $L_0$ and $L_1$ coincide, and their orientation is opposite. This implies $\tilde{\gamma}_2(1)\neq\tilde{\gamma}_2(0)$, which is a contradiction.
We proved that the braid $b$ is periodic with period $3$ or a power of $\Delta_3^{2\ell}$. \hfill $\Box$

\begin{lemm}\label{lem10.30}
Let $A=\{z\in \mathbb{C}: r_1<|z|<r_2\}$ be an annulus and
$g:A\to C_3(\mathbb{C})\diagup \mathcal{S}_3$  a holomorphic mapping that extends holomorphically to a neighbourhood of the closure $\bar A$ and represents a conjugacy class $\hat b$ of braids.
Suppose there exists smooth arc in a neighbourhood of $\bar A$ that intersects $ A$ along an arc $L_0$, 
with endpoints on different boundary circles, such that $g(L_0)\subset \mathcal{H}$.

Take any point $q\in L_0$ and a positively oriented dividing curve $\alpha$ in $A$ with initial and terminal points equal to $q$ that does not intersect $L_0$ at any other point.
Let $\mathfrak{A}$ be a complex linear mapping for which $\mathfrak{A}(q)=\{0,x,1\}\in  C_3(\mathbb{R})\diagup \mathcal{S}_3$ with $x\in(0,1)\,.\;$
Denote by $b_{tr}$ the element of $\pi_1( C_3(\mathbb{C})\diagup \mathcal{S}_3,C_3(\mathbb{R})\diagup \mathcal{S}_3)$ represented by $\mathfrak{A}(g)\mid \alpha$, and by $b$ the respective braid.
Then $b\in \hat b$ and $\Lambda(b_{tr})\leq \lambda(A)$.

If $\hat b$ is represented by non-periodic braids and $g$ intersects $\mathcal{H}$ transversally, then an arc $L_0$ as above exists.
\end{lemm}

\noindent {\bf Proof.}
Since the mapping $g$ is free homotopic to the mapping  $\mathfrak{A}(g)$ for any complex affine mapping $\mathfrak{A}$, we may assume that for the mapping $g$ itself $g(q)\in C_3(\mathbb{R})\diagup \mathcal{S}_3$, and, moreover,  $g(q)=\{0,x,1\}$.
The set $A\setminus L_0$ can be considered as curvilinear rectangle with curvilinear horizontal sides being copies of $L_0$. 
Let $\omega(z):A\setminus L_0\to R$ be the conformal mapping of the curvilinear rectangle onto a rectangle of the form $R=\{z\in \mathbb{C}:{\rm Re}z \in (0,1),\, {\rm Im}z\in (0,{\sf{a}})\}$, that takes the curvilinear side of $A\setminus L_0$, that is attained by moving in $A\setminus L_0$ clockwise towards
 $L_0$, to the lower side of $R$. (Note that the number $\sf{a}$ is uniquely defined by $A\setminus L_0$.)

Consider the restriction $g\mid A\setminus L_0$ and take the lift
$\tilde g=(\tilde {g}_1,\tilde {g}_2,\tilde {g}_3):A\setminus L_0 \to \mathcal{C}_3(\mathbb{C})$ of $g$
to a mapping  from $A\setminus L_0$ to the configuration space $\mathcal{C}_3(\mathbb{C})$, for which the clockwise continous extension to $q$ equals $(0,x,1)$ with $x\in(0,1)$.
For $z\in  A\setminus L_0  $ we consider  the complex affine mapping
\begin{align*}
\mathfrak{A}_{z}(      \zeta)=\mathfrak{a}(z)\zeta +\mathfrak{b}(z)\stackrel{def}= \frac{\zeta-\tilde{g}_1(z)}{\tilde{g}_3(z)-\tilde{g}_1(z)}, \, \zeta \in \mathbb{C}\,.
\end{align*}
We define the holomorphic mapping
\begin{align*}
\hat{g}(z)\stackrel{def}=\mathfrak{A}_{z}(\tilde{g}(z))=(0, \frac{\tilde{g}_2(z)-\tilde{g}_1(z)}{\tilde{g}_3(z)-\tilde{g}_1(z)},1)
,\,z \in A\setminus L_0 \,,
\end{align*}
and put $\hat{g}_{\rm sym}(z)\stackrel{def}=   \mathcal{P}_{\rm sym}(\hat{g})(z)=\mathfrak{A}_{z}(g(z))=
\{0, \frac{\tilde{g}_2(z)-\tilde{g}_1(z)}{\tilde{g}_3(z)-\tilde{g}_1(z)},1\}, z\in A\setminus L_0  $.

The mapping $\hat{ g}_{\rm sym}$ extends continuously to the horizontal sides of $ A\setminus L_0$. Since for a triple of complex numbers $\{Z_1,Z_2,Z_2\}$ the ratio $\frac{Z_2-Z_1}{Z_3-Z_1}$ is real if and only if
the points $Z_1, Z_2$, and $Z_3$ lie on a real line and $g$ maps the horizontal sides of $A\setminus L_0$ to $\mathcal{H}$, the mapping $\hat {g}_{\rm sym}$ takes the horizontal sides of $A\setminus L_0$ into $ C_3(\mathbb{R})\diagup \mathcal{S}_3$.  Hence, the mapping $\hat {g}_{\rm sym}$ on $A\setminus L_0  $ represents an element $b'_{tr}\in \pi_1( C_3(\mathbb{C})\diagup \mathcal{S}_3,C_3(\mathbb{R})\diagup \mathcal{S}_3)$. We obtain
$\Lambda(b'_{tr})\leq \lambda(  A\setminus L_0)\,.$
By Corollary \ref{cor4.0}
$$
\Lambda(b'_{tr})\leq \lambda(A)\,.
$$

By construction
the mappings $t\to g(\alpha (t))$ and $t\to \hat{ g}_{\rm sym}(\alpha (t))$, $t\in[0,1]$, take the same value at $0$.
Moreover, for each $t\in (0,1)$ the value $\hat{ g}_{\rm sym}(\alpha (t))$ is obtained by applying to the value $ g(\alpha (t))$ a complex linear mapping, 
i.e.  $\hat{ g}_{\rm sym}(\alpha (t))=  \mathfrak{a}(\alpha(t))\cdot g(\alpha(t)) + \mathfrak{b}(\alpha(t)),\, t\in [0,1]$,  for continuous functions $ \mathfrak{a}$ and $ \mathfrak{b}$ with $ \mathfrak{a}$ nowhere vanishing, $ \mathfrak{b}(\alpha(0))=0$, $ \mathfrak{a}(\alpha(0))=1$, and $ \mathfrak{b}(\alpha(1))$ and $ \mathfrak{a}(\alpha(1))$ real valued.
The function $ \mathfrak{b}\circ\alpha:[0,1]\to \mathbb{C}$ is homotopic with endpoints in $\mathbb{R}$ to the function that is identically equal to zero. The mapping $ \mathfrak{a}\circ\alpha:[0,1]\to \mathbb{C}\setminus \{0\}$ is homotopic with endpoints in $\mathbb{R}$ to $\frac{ \mathfrak{a}\circ\alpha}{| \mathfrak{a}\circ\alpha|}$.
Hence, the mappings $\hat{ g}_{\rm sym}({\alpha})$ and $\frac{ \mathfrak{a}(\alpha)}{| \mathfrak{a}(\alpha)|}{g}({\alpha})$ from $[0,1]$ to $C_3(\mathbb{C})\diagup \mathcal{S}_3$ are homotopic with endpoints in $C_3(\mathbb{R})\diagup \mathcal{S}$.

There is a continous function $\beta$ on $[0,1]$ such that  $\frac{\mathfrak{a}\circ\alpha(t)}{|\mathfrak{a}\circ\alpha(t)|}=e^{i\beta(t)}$ and $\beta(0)=0$. Then $\beta(1)=k\pi $ for an integer number $k$. 
Put $\mathring{g}(z)= e^{-\frac{\pi k i z}{\sf a}} \cdot \hat{ g}_{\rm sym}(z),\, z\in R$.
Since the curves $t\to \frac{\mathfrak{a}\circ\alpha(t)}{|\mathfrak{a}\circ\alpha(t)|}$ and $t\to e^{i\beta(t)},\, t\in[0,1]$, are homotopic with fixed endpoints, the
restrictions of $\mathring{ g}\mid \alpha$ and $g\mid\alpha$ are homotopic. Hence, the equality  $(b'\Delta_3^{-k})_{tr}=b_{tr}$ holds,
and by Lemma \ref{lemm1}
$$
\Lambda(b_{tr})\leq \lambda(A)\,.
$$
To prove the last statement of the lemma we notice that
the set $\{z\in A: g(z)\in \mathcal{H}\}$ is a smooth manifold of codimension $1$ in a neighbourhood of $\bar A$. If it did not contain an arc that joins the two boundary circles there would exist a simple closed dividing curve for the annulus whose image under $g$ avoids $\mathcal{H}$. This would contradict Lemma \ref{lemm3-braids2}.
The lemma is proved.
\hfill $\Box$

\chapter{Counting functions}
\label{Ch9}
\setcounter{equation}{0}

The first result that states exponential growth of the entropy counting function is due to Veech \cite{Ve}. More precisely, he
considered
conjugacy classes of pseudo-Anosov elements of the mapping class group of a closed  Riemann surface $S$, maybe with distinguished points $E_n$, with hyperbolic universal covering of $S\setminus E_n$, and proved that the number of classes  with entropy not exceeding a positive number $Y$, grows exponentially in $Y$.   Notice that the $3$-braids with positive entropy are exactly the $3$-braids that correspond to pseudo-Anosov elements of the mapping class group of $\,\mathbb{P}^1$ with fixed point $\infty$ and three other distinguished points.
The precise asymptotic for the entropy counting function is given by Eskin and Mirzakhani \cite{EsMi} for closed Riemann surfaces of genus at least $2$. Both papers, \cite{Ve} and \cite{EsMi}, use deep techniques of Teichm\"uller theory, in particular, these papers are based on the study of the Teichm\"uller flow.
\index{Veech} \index{Eskin} \index{Mirzakhani}

We give special attention to the growth estimates for the counting function related to the extremal length $\Lambda_{tr}$ of braids. The reason is that
the extremal length with totally real boudary values is more flexible for applications to  Gromov's Oka Principle than the entropy (equivalently the extremal length of conjugacy classes of braids). In particular, estimates of  the counting function related to the extremal length $\Lambda_{tr}$ allow to obtain effective finiteness theorems in the spirit of the Geometric Shafarevich Conjecture for the case when the base manifold is of second kind (\cite{Jo4}). We will address this application of the concept of extremal length in Chapter \ref{chapterfin}. In this chapter we obtain as a corollary of the results on the the counting function related to $\Lambda_{tr}$ an alternative proof of the exponential growth of the number of conjugacy classes of elements of $\mathcal{B}_3\diagup\mathcal{Z}_3$ that have positive entropy not exceeding  $Y$. Our proof does not use deep techniques from Teichm\"uller theory.
\index{Geometric Shafarevich Conjecture}

Notice that the bounds in the theorems of the present chapter can be improved. 
For the sake of simplicity of the proof we restrict ourselves to these estimates.

The first theorem of this chapter treats the
pure braid group modulo its center, ${\mathcal{PB}_3}\diagup \mathcal{Z}_3$. Recall that ${\mathcal{PB}_3}\diagup \mathcal{Z}_3$ is a free group in two generators $\sigma_j^2\diagup \mathcal{Z}_3$, $j=1,2$, which is isomorphic to the fundamental group $\pi_1(\mathbb{C}\setminus \{-1,1\},0)$ with standard generators $a_j$ (see Lemma \ref{lemm1a}).
The isomorphism takes the generators  $\sigma_j^2\diagup \mathcal{Z}_3$
to the generators $a_j$.

By Lemma \ref{lemm1} the quantities  $\Lambda_{tr}(\mathbold{b})$ and
$\Lambda(\hat{\mathbold{b}})$ are well defined by $\Lambda_{tr}(\mathbold{b})=\Lambda_{tr}({b}) $
 and $\Lambda(\hat{\mathbold{b}})= \Lambda(\hat{{b}})$ for any $3$-braid $b$ representing the class
$\mathbold{b}$.

The counting function $N^{\Lambda}_{\mathcal{PB}_3}(Y),\, Y \in (0,\infty),$ is defined as follows. For each positive parameter $Y$ the value of $N^{\Lambda}_{\mathcal{PB}_3}(Y)$ is equal to the number of elements
$\mathbold{b} \in \mathcal{PB}_3\diagup \mathcal{Z}_3$ 
with $0 < \Lambda_{tr}(\mathbold{b})\leq Y$. Note that here we count elements of $\mathcal{PB}_3\diagup \mathcal{Z}_3$ rather than conjugacy classes of such elements. We wish to point out that (with the mentioned choice of the generators) the condition $\Lambda_{tr}(\mathbold{b})>0$ excludes exactly the classes in  $\mathcal{PB}_3\diagup \mathcal{Z}_3$ of the even powers $\sigma_j^{2k}$ of the standard generators of the braid group $\mathcal{B}_3$ (see Theorem \ref{thm1} in Section \ref{sec:3-braids1}).
Notice that $N^{\Lambda}_{\mathcal{PB}_3}=N^{\Lambda}_{\mathbb{C}\setminus\{-1,1\}}$ for the counting function $N^{\Lambda}_{\mathbb{C}\setminus\{-1,1\}}(Y)$, that counts the number of elements of the fundamental group
$\pi_1(\mathbb{C}\setminus \{-1,1\},0)$ with positive $\Lambda_{tr}$ smaller than $Y$.
 \index{$N^{\Lambda}_{\mathcal{PB}_3}(Y)$}

\begin{thm}\label{thm10.1} For all positive numbers $Y\geq 12 \pi $ the inequality
\begin{equation}\label{eq10.1}
\frac{1}{2} \exp( \frac{\log 2}{6\pi}Y)
\leq N^{\Lambda}_{\mathcal{PB}_3}(Y)\leq  \frac{1}{2}e^{6 \pi Y}
\end{equation}
holds. The upper bound is true for all $Y>0$.
\end{thm}

Consider arbitrary $3$-braids. The counting function $N^{\Lambda}_{\mathcal{B}_3}(Y),\, Y \in (0,\infty),$ \index{$N^{\Lambda}_{\mathcal{B}_3}(Y)$} is defined as the number of elements  $\mathbold{b}\in \mathcal{B}_3 \diagup \mathcal{Z}_3$ 
with $0 < \Lambda_{tr}(\mathbold{b})\leq Y$. Note that the condition $\Lambda_{tr}(\mathbold{b})>0$ excludes exactly the elements $\mathbold{b}\in \mathcal{B}_3 \diagup \mathcal{Z}_3$ that are represented by $b=\sigma_j^k \Delta_3^{\ell}$
for $j=1$ or $2$, and $\ell = 0$ or $1$. (See Lemma \ref{lemm1'} and Theorem \ref{thmbr.3}.)

\begin{thm}\label{thm10.2} For any positive number $Y\geq 12 \pi $ 
the inequality
\begin{equation}\label{eq10.12}
\frac{1}{2} \exp( \frac{\log 2}{6\pi}Y)
\leq N^{\Lambda}_{\mathcal{B}_3}(Y)\leq 4 e^{6 \pi Y}
\end{equation}
holds. The upper bound is true for all $Y > 0$.
\end{thm}

The entropy counting function $N^{entr}_{\mathcal{B}_n}(Y),\, Y>0,$ for $n$-braids is defined 
as the number of conjugacy classes of pseudo-Anosov elements of $\mathcal{B}_n \diagup \mathcal{Z}_n$ with entropy not exceeding $Y$. For $3$-braids the value $N^{entr}_{\mathcal{B}_3}(Y),\, Y>0$, is also equal to the number of  conjugacy classes of elements of $\mathcal{B}_3 \diagup \mathcal{Z}_3$ with positive entropy not exceeding $Y$. Indeed, by Lemma \ref{lemEl.0}
the reducible elements of $\mathcal{B}_3\diagup \mathcal{Z}_3$  are exactly the conjugates of powers of $\sigma_1 \diagup \mathcal{Z}_3$. By Example 5 of Section \ref{sec:4.1b} their entropy equals zero. Since periodic elements of $\mathcal{B}_3\diagup \mathcal{Z}_3$ have zero entropy, the pseudo-Anosov $3$-braids are exactly the $3$-braids of positive entropy. \index{$N^{entr}_{\mathcal{B}_n}(Y)$} \index{entropy counting function} 
The following theorem will be obtained as a corollary of the Main Theorem and Theorem \ref{thm10.2}.

\begin{thm}\label{thm10.0} 
For any number $Y \geq \frac{70}{3}
\pi^2$ the estimate
\begin{equation}\label{eq10.0}
\frac{1}{4} e^{\frac{\log 2}{6\pi^2}\,Y}   \leq N^{entr}_{\mathcal{B}_3}(Y)\leq 4e^{12 Y}
\end{equation}
holds. The upper bound holds for all positive $Y$.
\end{thm}

We will prove now Theorem \ref{thm10.1} using the
isomorphism from  $\mathcal{PB}_3 \diagup \mathcal{Z}_3$ onto 
$\pi_1(\mathbb{C} \setminus\{-1,1\},0)$ that takes for $j=1,2$ the generator  $\sigma_j^2\diagup \mathcal{Z}_3$
to the generator $a_j$. 
An element of $\mathbold{b}\in\mathcal{PB}_3 \diagup \mathcal{Z}_3\cong \pi_1(\mathbb{C} \setminus\{-1,1\},0)$ will often be identified with the
reduced word in the $a_j \cong \sigma_j^2\diagup \mathcal{Z}_3$ that represents $\mathbold{b}$.

The proof of Theorem \ref{thm10.1} uses the syllable decomposition of words $w \in \pi_1(\mathbb{C}\setminus \{-1,1\},0)$ (see Definition \ref{defnxxx}). 
Recall that for a word $w \in \pi_1(\mathbb{C}\setminus \{-1,1\},0)$ in reduced form
$w= a_{j_1}^{k_1}\, a_{j_2}^{k_2} \, \ldots\,$ 
each term $a_{j_i}^{k_i}$ with $|k_i|\geq 2$ is a syllable of form $(1)$. Further, any maximal sequence of consecutive terms  $a_{j_i}^{k_i}$,  for which $|k_i|=1$ and all $k_i$ have the same sign, is a syllable of form (2) 
(if the number of terms is bigger than $1$) or of form $(3)$ (if the number of terms is equal to $1$). This gives a uniquely defined decomposition into syllables. Recall that the degree or length of the syllable is the sum of absolute values of the exponents of terms appearing in the syllable. We recall also the convention that the number of syllables of the identity equals zero. 

Recall (see \eqref{eq3+}) that we associated to each word $w$ in the $a_j$ the value  $\mathcal{L}(w)\stackrel{def}=   \sum_j \log(2 d_j+\sqrt{4 d_j^2-1})$, where the sum runs over the degrees $d_j$ of all syllables of the word.
For convenience we estimate the value $\mathcal{L}(w)$ from below 
by a slightly simpler value. More precisely, for
a non-trivial word $w \in \pi_1(\mathbb{C}\setminus \{-1,1\},0) \cong \mathcal{PB}_3\diagup \mathcal{Z}_3$    we put $\mathcal{L}_-(w)\stackrel{def}= \sum \log(3 d_k)\leq \mathcal{L}(w)=\sum\log(2d_k+\sqrt{4d_k^2-1})$, 
where each sum runs over the degrees $d_k$ of all syllables of $w$.

The main ingredient of the proof 
is the following lemma.

\begin{lemm}\label{lem10.1a}
Let $N^{\mathcal{L}_{-}}_{\mathcal{PB}_3}$ be the function whose value at any $Y >0$ is the number of reduced words  $w \in \pi_1(\mathbb{C}\setminus \{-1,1\},0)$, $w \neq \mbox{Id}$, for 
which $\mathcal{L}_{-}(w)\leq Y$. The following inequality
\begin{equation}\label{eq10.1a'}
N^{\mathcal{L}_-}_{\mathcal{PB}_3}(Y) \leq \frac{1}{2}e^{3Y}
\end{equation}
holds.
\end{lemm}
We need some preparation for the proof of Lemma \ref{lem10.1a}. 
Consider all finite tuples $(d_1,\ldots,d_j)$,
where $j\geq 1$ is any natural number (depending on the tuple) and the $d_k \geq 1$ are natural numbers.
Before proving Lemma \ref{lem10.1a} we estimate for $Y>0$ the number $N^*(e^Y)$ of different ordered tuples $(d_1, d_2,\ldots,d_j)$ (with varying $j$) that may serve as the degrees of the syllables of words (counted from left to right) with $\sum_1^j \log(3 d_k) \leq Y$. 
Put $X=e^Y$. Then $N^*(X)$ is the number of distinct tuples 
with $\prod (3d_k)\leq X$.

Fix a natural number $j$. Denote by $N_j^*(X),\, X \geq 1,$  the number of tuples $(d_1,\ldots,d_j)$ for which $\prod_1^j (3d_k) \leq X$. The number is not zero if and only if $j \leq \frac{\log X}{\log 3}$. For $X \geq 1$ the equality
\begin{equation}\label{eq10.4}
N^*(X) = \sum_{j \in \mathbb{Z}: \,1\leq j\leq  \frac{\log X}{\log 3}} N_j^*(X)
\end{equation}
holds.

Notice first that
$N_1^*(X)= [\frac{ X}{ 3}]$, where $[x]$ denotes the largest integer not exceeding the positive number $x$. Indeed we are looking for the number of $d_1$'s for which $1 \leq d_1 \leq \frac{ X}{ 3}$.

The value of $N_j^*$ for $j \geq 2$ is estimated by the following lemma.
Notice that the lemma holds also for $j=1$ if in the inequality
\eqref{eq10.7} we define $ 0! \stackrel{def}= 1$.

\begin{lemm}\label{lem10.2}
Let $j\geq 2$. Then $N_j^*(X)=0$ for $X< 3^j$. For  $X\geq 3^j$
\begin{equation}\label{eq10.7}
N_j^*(X) \leq \frac{1}{(j-1)!} \frac{1}{3} (\frac{2}{3})^{j-1} X \Big(\log \big(\frac{1}{3} (\frac{2}{3})^{j-1} X\big)\Big)^{j-1}.
\end{equation}
\end{lemm}

\noindent {\bf Proof.}
The number $N_2^*(X)$ is the number of tuples $(d_1,d_2)$ for which $3d_1 \cdot 3d_2 \leq X$. Since $d_1 d_2 \geq 1$, $N_2^*(X)=0$ for $X<3^2$.  If
$X\geq 3^2$ the inequality $d_1 \leq \frac{X}{9}$ holds, and for given $d_1$ the number $d_2$ runs through all natural numbers with $1\leq d_2 \leq \frac{X}{9d_1}$. Hence, for $X\geq 3^2$

\begin{equation}\label{eq10.5}
N_2^*(X) \leq  \sum_{k \in \mathbb{Z}: \, 1 \leq k \leq  \frac{X}{3^2}} \frac{X}{3^2 k}.
\end{equation}
Put $a=\frac{X}{3^2}$ and $k'= \frac{k}{a}$. Since for positive numbers $k'$ and $\alpha$ with $k'>\alpha$ the inequality $\frac{1}{k'}\leq \frac{1}{\alpha} \int_{k'-\alpha}^{k'} \frac{dx}{x}$ holds, we obtain
\begin{align}\label{eq10.6}
N_2^*(X) \leq   \sum_{k' \in \frac{1}{a}\mathbb{Z}:\, \frac{1}{a} \leq k' \leq  1 } \frac{1}{k'}
\leq  2a \int_{\frac{1}{2a}}^1 \frac{dx}{x} = 2a \log(2a) =
\frac{1}{3} \cdot \frac{2}{3}X  \cdot \log(\frac{1}{3}\cdot \frac{2}{3}X ).
\end{align}

For $j\geq 3$ we provide induction using the following fact. Let $0<x<1$. Then for any positive integer $j$ the value
$(\frac{(-\log x)^j}{x})'= -\frac{(-\log x)^{j-1}}{x^2}(j-\log x) $ is negative.
Hence, for $k'\in (0,1)$ and $0<\alpha < k'$ 
\begin{equation}\label{eq10.6a}
\frac{1}{k'}(-\log(k'))^j  \leq \frac{1}{\alpha} \int_{k'-\alpha}^{k'} \frac{(-\log x)^j}{x}dx \,.
\end{equation}
We saw that the lemma is true for $j=2$. Suppose it is true for $j$. Prove that then it holds for $j+1$.  The number $N_{j+1}^*(X)$ is the number of tuples $(d_1,\ldots,d_{j+1})$ for which $3d_1 \cdot \ldots \cdot 3d_{j+1} \leq X$. Hence, $N_{j+1}^*(X)=0$ if $X <3^{j+1}$. If $X \geq 3^{j+1}$, then 
$d_1 \leq \frac{X}{3^{j+1}}$ and for given $d_1$ the tuples $(d_2,\ldots,d_{j+1})$ run through all tuples with $3 d_2 \cdot \ldots \cdot 3 d_{j+1} \leq \frac{X}{3d_1}$. Hence, for $X \geq 3^{j+1}$

\begin{align}\label{eq10.8}
N_{j+1}^*(X) =  & \sum_{k \in \mathbb{Z}:\, 1 \leq k \leq  \frac{X}{3^{j+1}}} N_j^*(\frac{X}{3k})\nonumber\\
\leq & \sum_{k \in \mathbb{Z}:\, 1 \leq k \leq  \frac{X}{3^{j+1}}} \frac{1}{(j-1)!}  \frac{1}{3} (\frac{2}{3})^{j-1}  \frac{X}{3k} \Big( \log\big(\frac{1}{3} (\frac{2}{3})^{j-1}  \frac{X}{3k}\big)\Big)^{j-1}.
\end{align}

Put $a= \frac{1}{9} (\frac{2}{3})^{j-1} X$ and $k'= \frac{k}{a}$. Then
\begin{align}\label{eq10.9}
N_{j+1}^*(X) \leq  & \frac{1}{(j-1)!}\sum_{k' \in \frac{1}{a}\mathbb{Z}:\frac{1}{a} \leq k' \leq  \frac{1}{2^{j-1}}} \frac{1}{k'} \big(\log(\frac{1}{k'})\big)^{j-1}\nonumber \\
\leq &  \frac{1}{(j-1)!} 2a \int_{\frac{1}{2a}}^{\frac{1}{2^{j-1}}} \frac{1}{x} (-\log x )^{j-1} dx
\leq   \frac{1}{(j-1)!} 2a \int_{\frac{1}{2a}}^1 \frac{1}{x} (-\log x )^{j-1} dx
\nonumber\\
= & (-1)^{j-1} \frac{1}{(j-1)!} 2a \frac{1}{j}(\log x)^j\mid_{\frac{1}{2a}}^{1}.
\end{align}
We obtain
\begin{align}\label{eq10.10}
N_{j+1}^*(X)
\leq & \frac{1}{j!} 2a \big(\log (2a)\big)^j = \frac{1}{j!}\frac{1}{3} (\frac{2}{3})^{j}X \Big(\log\big( \frac{1}{3} (\frac{2}{3})^{j}X\big)\Big)^j.
\end{align}
Lemma \ref{lem10.2} is proved.
\hfill $\Box$

\bigskip

Lemma \ref{lem10.2} implies the following upper bound for $N^*$.

\begin{lemm}\label{lem10.1}
For $X<3$ the function $N^*(X)$ vanishes. Moreover, for any positive number X
\begin{equation}\label{eq10.0a}
N^*(X) \leq (\frac{X}{3})^{\frac{5}{3}}.
\end{equation}
\end{lemm}

\noindent {\bf Proof of Lemma \ref{lem10.1}.}
Since all $N_j^*(X)$ vanish for $X<3$ the value $N^*(X)$ vanishes for such $X$. 
For $X\geq 3$ equation \eqref{eq10.4} and Lemma \ref{lem10.2} imply
\begin{align}\label{eq10.11}
N^*(X) \leq & \sum_{j \in \mathbb{Z}: \, 1 \leq j \leq \frac{\log X}{\log 3}}  \frac{1}{(j-1)!} \frac{1}{3} (\frac{2}{3})^{j-1} X \Big(\log\big(\frac{1}{3} (\frac{2}{3})^{j-1} X\big)\Big)^{j-1} \nonumber \\
\leq & \; \frac{1}{3} X \sum_{j \in \mathbb{Z}: \, 1 \leq j \leq \frac{\log X}{\log 3}} \frac{1}{(j-1)!} \big(\frac{2}{3} \log (\frac{1}{3}X)\big)^{j-1} \nonumber \\  \leq & \; \frac{1}{3} X \exp\big(\frac{2}{3}\log(\frac{1}{3}X)\big)=
(\frac{X}{3})^{\frac{5}{3}}.
\end{align}
Lemma \ref{lem10.1} is proved. \hfill $\Box$

\bigskip

\noindent {\bf Proof of Lemma \ref{lem10.1a}.} 
We assume that $[\frac{Y}{\log 3}]\geq 1$. Otherwise $N^*(e^Y)$ vanishes, and therefore $N^{\mathcal{L}_-}_{\mathcal{PB}_3}(Y)=0$, and the inequality is satisfied.
We will use the notation $N_j^{\mathcal{L}_-}(Y), \, j\geq 1,$ for the number of different reduced words $w$ in $\pi_1(\mathbb{C} \setminus \{-1,1\}, 0)$ that  consist of $j$ syllables and satisfy the inequality $\mathcal{L}_-(w)= \sum_{k=1}^j \log(3d_k) \leq  Y$.
Then $N^{\mathcal{L}_-}_{\mathcal{PB}_3}(Y)=  \sum_{j=1}^{j_0}N_j^{\mathcal{L}_-}(Y)$ with $j_0\stackrel{def}=[\frac{Y}{\log 3}]$. 
We will estimate $N_j^{\mathcal{L}_-}(Y)$ by $N_j^*(X)$ with $X= e^{ Y}$. Recall that $N^*(e^Y)$ is the number of different tuples $(d_1,\ldots,d_j)$ with $d_k \geq 1$ for which $\prod_{k=1}^j (3d_k)\leq e^{ Y}$.

For this purpose we take a tuple $(d_1,\ldots,d_j)$ and estimate the number of different reduced words with tuple of lengths of syllables (from left to right) equal to  $(d_1,\ldots,d_j)$. The first syllable is completely determined by its first letter  (which may be $a_1^{\pm 1}$ or $a_2^{\pm 1}$), its length, and the fact which of the following two options hold: either it is of form (1), or of form (2) or (3). 
Hence, there are at most $8$ different choices for the first syllable if we require the syllable to have exactly degree $d_1$. For all other syllables
the first letter of the syllable cannot be $a_i^{\pm 1}$ if the last letter in the preceding syllable is $a_i^{\pm 1}$ for the same $a_i$. Hence, for all but the first syllable there are at most $4$ choices given the degree of the syllable and the preceding syllable.

We showed that for all $j=1,\ldots, j_0 =[\frac{Y}{\log 3}]$, and each tuple $(d_1,\ldots,d_j)$,  there are at most $2 \cdot 4^j$ different reduced words with tuple of lengths of syllables  equal to  $(d_1,\ldots,d_j)$. Hence, for $Y  \geq \log 3$ the number $N^{\mathcal{L}_-}_{\mathcal{PB}_3}(Y)$ of reduced words
$w \in   \pi_1(\mathbb{C}\setminus \{-1,1\},0)$, $w\neq \mbox{Id}$,  with $\prod_1^{j_0} (3d_k) \leq \exp(Y)$ equals 
\begin{equation}\label{eq10.3}
N^{\mathcal{L}_-}_{\mathcal{PB}_3}(Y)= \sum_{j=1}^{j_0}N_j^{\mathcal{L}_-}(Y)\;\,\leq\;\;
\sum_{j=1}^{j_0} 
2\cdot 4^{j_0}N_j^*(e^{Y})\, =\,
2\cdot 4^{j_0} \cdot N^*(e^{Y}).
\end{equation}

Using the inequality $j_0 
\leq \frac{Y}{\log 3}$ and Lemma \ref{lem10.1} with $X=e^{Y}$ we obtain
the requested estimate by the value
\begin{equation}\label{eq10.3a}
2\cdot 3^{-\frac{5}{3}} \exp((\frac{ \log 4}{\log 3}+\frac{5}{3}) Y)< \frac{1}{2}\exp(3 Y).
\end{equation}
We used the inequalities $(\frac{ \log 4}{\log 3}+(\frac{5}{3}))<2.93$ and $4 \cdot 3^{-\frac{5}{3}} < 0.65<1$. 
Lemma \ref{lem10.1a} is proved.
\hfill $\Box$

\bigskip

\noindent {\bf Proof of 
Theorem \ref{thm10.1}.}  By Theorem \ref{thm1} the inequality  
$ \frac{1}{2 \pi}\mathcal{L}_-(w)\leq\frac{1}{2 \pi} \mathcal{L}(w)    \leq \Lambda_{tr}(w) $ 
holds for all reduced words $w$ representing elements in $\mathcal{PB}_3\diagup \mathcal{Z}_3\cong \pi_1(\mathbb{C}\setminus \{-1,1\},0)$ that are not equal to a power of $a_1$ or of $a_2$  or to the identity (equivalently, for which $\Lambda_{tr}(w)>0$). 
This inequality implies the inclusion
$\{w:0<\Lambda_{tr}(w) \leq Y\} \subset \{w \neq \mbox{Id}: \mathcal{L}_-(w) \leq 2\pi Y \}$. We obtain
the inequality
$N_{\mathcal{PB}_3}^{\Lambda}(Y) \leq N^{\mathcal{L}_-}_{\mathcal{PB}_3}(2 \pi Y)$  and by
Lemma \ref{lem10.1a} the right hand side of this inequality does not exceed $\frac{1}{2}e^{6\pi Y}$. This gives the upper bound.

The lower bound is obtained as follows.
Consider all reduced words in $\pi_1(\mathbb{C}\setminus\{-1,1\},0)$ of the form
\begin{equation}\label{eq10.2''}
a_1^{2k_1} a_2^{2k_2} \ldots
\end{equation}
where each $k_i$ is equal to $1$ or $-1$. If $j$ is the number of syllables (i.e the number of  the $a_i^{k_i}$) of a word $w$ of the form \eqref{eq10.2''}, then 
by Proposition \ref{propxx} the inequality
$\Lambda_{tr}(w) \leq \frac{35}{6}\pi j$ 
holds.
Consider the words of the mentioned form for which
$j=j_0 \stackrel{def}=[\frac{Y}{{6}\pi}]$.  
Since $j_0$ must be at least equal to $2$ we get the condition
$Y\geq  12\pi$. 
For the chosen $j_0$  the extremal length of the considered words does not exceed $Y$. The number of different words of such kind equals 
$2^{j_0}= 2^ {[\frac{Y}{6 \pi}]}\geq \exp(\log 2 \cdot( \frac{Y}{6 \pi}-1))= \frac{1}{2} \exp( \frac{\log 2}{6\pi}Y)$. 
Theorem \ref{thm10.1} is proved. \hfill $\Box$

\medskip

Consider now arbitrary elements of the braid group modulo its center $\mathcal{B}_3\diagup \mathcal{Z}_3$ and their extremal length with totally real horizontal boundary values.

\medskip

\noindent {\bf Proof of Theorem \ref{thm10.2}.} 
We will use the fact that by Lemma \ref{lemm1'} any $b\in \mathcal{B}_3$ which is not a power of $\Delta_3$ can be uniquely written in the form
$\sigma_j^k \, b_1 \, \Delta_3^{\ell}\,$, where $k$ and $\ell$ are integers, $b_1$ is a word in $\sigma_1^2$ and $\sigma_2^2$ in reduced form whose first term is a non-zero  power of $\sigma_{j'}\neq \sigma_j$ if $b_1\neq {\rm id}$. For a braid in this form we defined
$\vartheta(b) =\sigma_j^{q(k)} \, b_1$
where $q(j)$ is that even integer neighbour
of $j$ (including perhaps $j$ itself), which is closest to zero.
For each element  $\mathbold{b}$
of $\mathcal{B}_3 \diagup \mathcal{Z}_3$ the value $\vartheta(\mathbold{b})$ is well defined by putting it equal to
$\vartheta({b})$ for any $b\in \mathcal{B}_3$ that represents $
\mathbold{b}$.
\index{$\vartheta(b)$}
\index{$\vartheta(\mathbold{b})$}

Take any element $\mathbold{b}$ of $\mathcal{PB}_3 \diagup \mathcal{Z}_3$. Choose its unique representative that can be written as a reduced word $w$ in $\sigma_1^2$ and $\sigma_2^2$. 
We describe now all elements $\mathbold{b}$
of $\mathcal{B}_3 \diagup \mathcal{Z}_3$ with 
$\vartheta(\mathbold{b})=w$. 
If $w\neq \mbox{Id}$ these are the elements represented by the following braids.
If the first term of $w$ is 
$\sigma_j^{2k}$ with $k\neq 0$, then the possibilities are ${b}=w \Delta_3^{\ell}$ with $\ell=0$ or $1$,  ${b}=\sigma_j^{{\rm{sgn}} k} w \Delta_3^{\ell}$ with $\ell=0$ or $1$, or  ${b}=\sigma_{j'}^{\pm 1} w \Delta_3^{\ell}$ with $\ell=0$ or $1$ and $\sigma_{j'}\neq
\sigma_{j}$. Hence, for $w\neq \mbox{Id}$ there are $8$ possible choices of elements $\mathbold{b}\in \mathcal{B}_3 \diagup \mathcal{Z}_3$ with $\vartheta(\mathbold{b})=w$. 
By Theorem \ref{thmbr.3} the set of $\mathbold{b} \in \mathcal{B}_3 \diagup \mathcal{Z}_3$ with $0 < \Lambda_{tr}(\mathbold{b}) \leq Y$ is contained in the set of  $\mathbold{b} \in \mathcal{B}_3 \diagup \mathcal{Z}_3$ with $\vartheta({\mathbold{b}}) = w \neq \mbox{Id}$, $ \mathcal{L}_-(w) \leq 2\pi Y$. We obtain
\begin{equation}\label{eq10.13}
N^{\Lambda}_{\mathcal{B}_3}(Y)\leq  8 N^{\mathcal{L}_-}_{{\mathcal{PB}}_3}(2\pi Y) .
\end{equation}
By Lemma \ref{lem10.1a} we obtain  $N^{\Lambda}_{\mathcal{B}_3}(Y)\leq 4 e^{6 \pi Y}$.

Since each pure $3$-braid is also an element of the braid group $\mathcal{B}_3$ the lower bound of Theorem \ref{thm10.1} provides also a lower bound for Theorem \ref{thm10.2}. Theorem \ref{thm10.2} is proved. \hfill $\Box$

\bigskip

\noindent {\bf Proof of the upper bound of Theorem \ref{thm10.0}.} 
For each conjugacy class $\hat{\mathbold{ b}}$ of elements of $\mathcal{B}_3\diagup\mathcal{Z}_3$ with $h(\hat{\mathbold{ b}})>0$ and each positive number $\varepsilon$
we will find a braid $b$ for which $\widehat{b\diagup \mathcal{Z}_3}=\hat{\mathbold{ b}}$ and
\begin{equation}\label{eq10.30}
\Lambda_{tr}(b)\leq \frac{2}{\pi}h(\hat{\mathbold{ b}})+\varepsilon\,.
\end{equation}
For this purpose we represent a conjugacy class $\hat b$ that represents $\hat{\mathbold{ b}}$ by a holomorphic map $g:A\to C_3(\mathbb{C})\diagup \mathcal{S}_3$ from an annulus $A$ of extremal length
\begin{equation}\label{eq10.29}
\lambda(A)<\frac{2}{\pi}h(\hat{\mathbold{b}})+\varepsilon
\end{equation}
to the symmetrized configuration
space. 
By the Holomorphic Transversality Theorem \cite{KZ} we may assume, after shrinking $A$ (keeping inequality \eqref{eq10.29}) and approximating $g$, that $g$ is holomorphic in a neighbourhood $A'$ of the closure $\bar A$ of $A$ and is transversal to  the smooth real hypersurface $\mathcal{H}$ (see equation \eqref{eq3-braids.4}). 
Since $h(\hat{\mathbold{b}})>0$, the representing braids of  $\hat{\mathbold{b}}$ are non-periodic and the conditions of Lemma \ref{lem10.30} 
are satisfied. Hence, there exists a braid $b$ for which $b\diagup\mathcal{Z}_3$ represents $\hat{\mathbold{b}}$ such that
$$
\Lambda(b_{tr})\leq \lambda(A)\,.
$$
By inequality \eqref{eq10.29} we obtain $\Lambda_{tr}(b)\leq \frac{2}{\pi} h(\hat{\mathbold{b}})+\varepsilon$. 
Using also Lemma \ref{lemm1}, we see that the number of conjugacy classes $\hat{\mathbold{b}}$ of $\mathcal{B}_3\diagup\mathcal{Z}_3$ of positive entropy not exceeding $Y$
does not exceed the number of elements  $\mathbold{b}\in \mathcal{B}_3\diagup\mathcal{Z}_3 $ with $\Lambda_{tr}(\mathbold{b})< \frac{2}{\pi} Y +\varepsilon$. In other words,
$N^{entr}_{\mathcal{B}_3}(Y)
 \leq N^{\Lambda}_{\mathcal{B}_3}(\frac{2}{\pi} Y +\varepsilon)$. Since for $\varepsilon$ we may take any a priory given positive number, Theorem \ref{thm10.2} implies
\begin{equation}\label{eq10.32}
N^{entr}_{\mathcal{B}_3}(Y)
 \leq N^{\Lambda}_{\mathcal{B}_3}(\frac{2}{\pi} Y) \leq 4 e^{12 Y}\,.
\end{equation}
We obtained the upper bound. \hfill $\Box$

\medskip

For obtaining the lower bound we need the following preparations.

\begin{lemm}\label{lem10.20}
Suppose $b_1$ and $b_2$ are elements of the free group $\mathbb{F}_n$ in $n$ generators, that are both represented by cyclically reduced words. Then $b_1 $ and $b_2$ are conjugate in $\mathbb{F}_n$ if and only if 
the word representing $b_2$ is obtained from the word representing $b_1$ by a cyclic permutation of terms.
\end{lemm}
\index{$\mathbb{F}_n$} \index{group ! free group}

\noindent {\bf Proof.} 
It is enough to prove the following statement. If under the conditions of the lemma
$b_2=w^{-1} b_1 w$ for an element $ w\in \mathbb{F}_n$,
$w\neq\mbox{Id}$, then $b_2=w'^{-1}b'_1 w'$ where $b'_1$ is represented by a cyclically reduced word that is obtained from the reduced word representing $b_1$ by a cyclic permutation of terms, and $w'$ is represented by a reduced word which has less terms than the word representing $w$.

This statement is proved as follows. Write $w=w_1 w'$ where $w_1\in \mathbb{F}_n$ is represented by the first term of the word representing $w$, and $b_1= a' B_1' a''$ where $a'$ and $a''$ are represented by the first and the last term, respectively, of the word
that represents $b_1$. Then $b_2=w'^{-1}w_1^{-1} a' B_1' a'' w_1 w'$. If both relations $w_1^{-1} a'\neq {\rm Id}$ and $a'' w_1\neq {\rm Id}$ were true, then the first and the last term of the  reduced word representing $b_2$ would be a power of the same generator of $\mathbb{F}_n$, which contradicts the fact that $b_2$ can be represented by a cyclically reduced word. If either $w_1^{-1} a'= {\rm Id}$ or $a'' w_1= {\rm Id}$, then the reduced words representing $b_1'\stackrel{def}=w_1^{-1} a' B_1' a'' w_1$ and $b_1$ are cyclic permutations of each other. Hence, the statement is true.
\hfill $\Box$

\begin{lemm}\label{lem10.21} The following equalities hold.
\begin{align}\label{eq10.16a}
\Delta_3 \sigma_1 = & \sigma_2 \Delta_3\,, \nonumber\\
\Delta_3 \sigma_2 = & \sigma_1 \Delta_3\,, \nonumber\\
\sigma_1^{-1}(\sigma_2^{-4}\Delta_3^4)\sigma_1= & \sigma_2^2 \sigma_1^2 \sigma_2^2 \sigma_1^2\,,\nonumber\\
\sigma_2^{-1}(\sigma_1^{-4}\Delta_3^4)\sigma_2= & \sigma_1^2 \sigma_2^2 \sigma_1^2 \sigma_2^2\,.
\end{align}
\end{lemm}

\noindent{\bf Proof.}
The third equality is obtained as follows
\begin{align}\label{eq10.16b}
& \sigma_1^{-1} (\sigma_2^{-4}\Delta_3^4)\sigma_1
=\sigma_1^{-1} \sigma_2^{-1}\Delta_3^2\sigma_2^{-2}\Delta_3^2\sigma_2^{-1}\sigma_1\nonumber\\
=&(\sigma_1^{-1}\sigma_2^{-1})(\sigma_2 \sigma_1 \sigma_2)(\sigma_2 \sigma_1 \sigma_2)\sigma_2^{-1}\sigma_2^{-1}(\sigma_2 \sigma_1 \sigma_2)(\sigma_2 \sigma_1 \sigma_2)\sigma_2^{-1}\sigma_1\nonumber\\
=&\sigma_2^2 \sigma_1^2 \sigma_2^2 \sigma_1^2\,.
\end{align}
The fourth equality is obtained by conjugating the third equality by $\Delta_3$. \hfill $\Box$

\medskip

Consider two elements $\mathbold{b}_1$ and $\mathbold{b}_2$ of
the pure braid group modulo center
$\mathcal{PB}_3\diagup \mathcal{Z}_3$.
In particular,  $\mathbold{b}_1$ and $\mathbold{b}_2$ are elements of the full braid group  modulo center
$\mathcal{B}_3\diagup \mathcal{Z}_3$.
We write  $a_1$ and $a_2$ for the generators of  $\mathcal{PB}_3\diagup \mathcal{Z}_3$.
\begin{lemm}\label{lem10.22} Suppose both elements $\mathbold{b}_1$ and $\mathbold{b}_2$ of $\,\mathcal{PB}_3\diagup \mathcal{Z}_3$  are of the form
\begin{equation}\label{eq10.15c}
a_1^{\pm 2}\, a_2 ^{\pm 2}\cdots a_1^{\pm 2}\, a_2 ^{\pm 2}\;\;\;\;\mbox{or}\;\;\;\; a_2^{\pm 2}\, a_1 ^{\pm 2}\cdots a_2^{\pm 2}\, a_1 ^{\pm 2}\,,
\end{equation}
with a positive number of terms,
and the reduced word representing $\mathbold{b}_1$ has at least four terms.
Then $\mathbold{b}_2$ cannot be conjugated to $\mathbold{b}_1$ by an element $\mathbold{\beta}$ of  $\mathcal{B}_3\diagup \mathcal{Z}_3$ that can be represented by a braid
$\beta=\sigma_j\beta_1 \Delta_3^{\ell}$
with some $j$ and $\ell$ and $\beta_1$ being a word in $\sigma_1^2$ and $\sigma_2^2$.
\end{lemm}
(See Lemma \ref{lemm1'} for writing a $3$-braid in a special form.)

\smallskip

\noindent{\bf Proof of Lemma \ref{lem10.22}.}
Indeed, suppose the contrary,
\begin{equation}\label{eq10.15a}
\mathbold{b}_1=\mathbold{\beta}^{-1} \mathbold{b}_2 \mathbold{\beta}
\end{equation}
with $\mathbold{\beta}$ represented by $\beta=\sigma_j\beta_1 \Delta_3^{\ell}$, 
where  $\ell=0,1,$ and $\beta_1\in \mathcal{PB}_3$ is a word in $\sigma_1^2$ and $\sigma_2^2$. For $j=1,2,$ we let $b_j$ be the representative of $\mathbold{b}_j$ which can be written as reduced word in $\sigma_1^2$ and $\sigma_2^2$.  (In other words, for the representing braids $b_j$ the first and third strand have linking number zero.)
By equation \eqref{eq10.16a} there is an integer number $n$ such that the braid $b_2'\stackrel{def}=\sigma_j^{-1} b_2 \sigma_j \Delta_3^{2n}$ is a product of a positive even number of factors which either have alternately the form $\sigma_1^{\pm 4}$ and $(\sigma_2^2 \sigma_1^2 \sigma_2^2 \sigma_1^2)^{\pm 1}$, or they have alternately the form $\sigma_2^{\pm 4}$ and $(\sigma_1^2 \sigma_2^2 \sigma_1^2 \sigma_2^2)^{\pm 1}$. The braid ${\beta_1} \Delta_3^{\ell} b_1 \Delta_3^{-\ell}\beta_1^{-1}$ can also be written as reduced word in $\sigma_1^2$ and $\sigma_2^2$. Hence, by equation \eqref{eq10.15a} the two braids $b_2'$ and ${\beta_1} \Delta_3^{\ell} b_1 \Delta_3^{-\ell}\beta_1^{-1}$ must be equal.

The element $b'_2\diagup {\mathcal Z}_3$ is the product of at least two factors equal to $a_1^{\pm 2}$ and $(a_1a_2a_1a_2)^{\pm 1}$ alternately, or it is the product of at least two factors equal to $a_2^{\pm 2}$ and $(a_2a_1a_2a_1)^{\pm 1}$ alternately.
Hence, the reduced word in $a_1$ and $a_2$ representing $b'_2$ contains at least $4$ terms, and each sequence of four consecutive terms of this word contains at least
two terms that appear with power $+1$ or $-1$.

Indeed, put $A=a_1a_2a_1a_2$. Replace in the product $A^{\ell_1}a_1^{\pm 2} A^{\ell_2}$ with $\ell_1=0,\pm 1,$ $\ell_2=0,\pm 1,$ each factor $A^{\ell_1}$ by the reduced word in $a_1$, $a_2$, representing it. We obtain a (possibly not reduced) word in $a_1$, $a_2$.

If $\ell_1=0$ or $\ell_1=1$, and $\ell_2=0$ or $-1$, the word is already reduced.

If $\ell_1=-1$, and $\ell_2=0$ or $-1$,
the reduced word representing the product, consists of the first three letters of the word representing $A^{-1}=A^{\ell_1}$, a non-trivial power of $a_1$ and all letters of the word representing  $A^{\ell_2}$.

If $\ell_1=0$ or $\ell_1=1$, and $\ell_2=1$,
the reduced word representing the product, consists of all letters of the word representing $A^{\ell_1}$, a non-trivial power of $a_1$ and the last three letters of the word representing  $A= A^{\ell_2}$.

Finally , if $\ell_1=-1$ and $\ell_2=1$ then the reduced word representing the product, consists of the first three letters of the word representing $A^{-1}=A^{\ell_1}$, a non-trivial power of $a_1$ and the last three letters of the word representing $A=A^{\ell_2}$.

Hence, the reduced word representing $b_2'\diagup {\mathcal Z}_3$ contains for each factor $A^{\ell_j}$ in the whole product at least the two middle letters of the word representing  $A^{\ell_j}$.
The case, when $b_2'\diagup {\mathcal Z}_3$ is the alternating product of  $a_2a_1a_2a_1$ and $a_2^2$ is obtained by replacing the role of the generators $a_1$ and $a_2$.

Since the element $(\Delta_3^{\ell} b_1 \Delta_3^{-\ell})\diagup {\mathcal Z}_3$  can be represented by a word of the form \eqref{eq10.15c} with at least four terms, 
we arrived at the following conjugation problem in the free group $\mathbb{F}_2$. We have an element $B_1\in \mathbb{F}_2$ of the from \eqref{eq10.15c} with at least $4$ terms, and an element $B_2\in \mathbb{F}_2$ written as reduced word with at least $4$ terms, such that each sequence of four consecutive terms of $B_2$
contains  at least two terms that appear with power $+1$ or $-1$. We have to prove that there is no element of $\mathbb{F}_2$ that conjugates $B_1$ to $B_2$.

To prove this statement we assume the contrary.
Suppose the reduced word $a_{j_1}^{k_1}\ldots a_{j_{\ell}}^{k_{\ell}}$ representing $B_2$ is not cyclically reduced. Identify each term $ a_{j_{\ell'}}^{k_{\ell'}}$ of the word with the element of $\mathbb{F}_2$ represented by it. 
The conjugate $(a_{j_1}^{k_1})^{-1}   B_2 a_{j_1}^{k_1}$ can be represented by a reduced word. If this reduced word is not cyclically reduced, then
the consecutive sequence of all its terms 
is also a consecutive sequence of terms of the reduced word representing $B_2$. Continue by induction in this way
with the $ a_{j_{\ell'}}^{k_{\ell'}}$ 
until we arrive at an element $B_2'\in \mathbb{F}_2$ that can be represented by a cyclically reduced word.
At each inductive step the consecutive sequence of all terms of the new reduced but not cyclically reduced word is a consecutive sequence of terms of the previously obtained word and, hence, also a consecutive sequence of terms of $B_2$.
Since $B_2$ is not the identity, the reduced word representing $B_2'$ is not the identity and is 
cyclically reduced.
The sequence of all its consecutive terms except the last one 
is also a sequence of consecutive terms of the reduced word representing $B_2$.

Since by our assumption $B_2'$ is conjugate to $B_1$ by an element of $\mathbb{F}_2$, and both,
$B_2'$ and $B_1$, are represented by cyclically reduced words, the representing words have the same number of terms  (see Lemma \ref{lem10.20} ). By the assumption for $B_1$ the number of terms is at least $4$. Hence, the words $B_2'$ and $B_2$ have at least $3$ consecutive terms in common. Therefore $B_2'$ contains a term that appears with power $\pm 1$ which contradicts Lemma \ref{lem10.20}. The lemma is proved. \hfill $\Box$

\medskip

\noindent {\bf Proof of the lower bound of Theorem \ref{thm10.0}.}
Consider the elements $\mathbold{b}\in\mathcal{PB}_3\diagup \mathcal{Z}_3 \subset \mathcal{B}_3\diagup \mathcal{Z}_3$ which can be represented by words of the form 
\begin{equation}\label{eq10.15}
w=a_1^{\pm 2}\, a_2 ^{\pm 2}\cdots a_1^{\pm 2}\, a_2 ^{\pm 2}
\end{equation}
in the generators $a_1\cong \sigma_1^2\diagup \mathcal{Z}_3$ and $a_2\cong \sigma_1^2\diagup \mathcal{Z}_3$ of $\mathcal{PB}_3\diagup \mathcal{Z}_3$ with at least $4$ terms.
We denote the number of syllables (in this case the number of the terms $a_i^{k_i}$) of the word \eqref{eq10.15} by $2j$.
Since each word of the form \eqref{eq10.15} is cyclically 
reduced, the Main Theorem  and Proposition \ref{propxxx}  imply that the entropy $h(\hat{\mathbold{b}})$ of the conjugacy class of the element $\mathbold{b}\in \mathcal{PB}_3\diagup \mathcal{Z}_3$ represented by $w$ satisfies the inequality
\begin{align}\label{eq10.16}
h(\hat{\mathbold{b}}) \leq \frac{\pi}{2}\cdot\frac{35}{6}\cdot 2\pi j
\end{align}
Take $j=j'_0 \stackrel{def}= [\frac{Y}{\frac{35}{6} \pi^2}]$. 
For this choice of $j'_0$ the inequality $h(\hat{\mathbold{b}}) \leq Y$ 
holds. Since we required that 
the number of syllables $2j'_0$ is at least $4$, we get the condition $Y \geq \frac{70}{3} \pi^2$. 
The number of different words of such kind equals $2^{2 j'_0}$.

We prove now that the number of different conjugacy classes of $\mathcal{B}_3\diagup\mathcal{Z}_3$ that can be represented by elements in $\mathcal{PB}_3$
corresponding to words \eqref{eq10.15} with $2j'_0$ syllables is not smaller than $\frac{2^{2j'_0}}{2j'_0}$. It is enough to prove the following claim. For each element
$\mathbold{b} \in \mathcal{PB}_3\diagup\mathcal{Z}_3$ of form \eqref{eq10.15} the number of
elements of $\mathcal{PB}_3\diagup\mathcal{Z}_3$ of form \eqref{eq10.15} that are conjugate to $\mathbold{b}$ by an element of $\mathcal{B}_3\diagup\mathcal{Z}_3$ does not exceed $2j'_0$.

Suppose two elements $\mathbold{b}_1$ and $\mathbold{b}_2$ of  $\mathcal{PB}_3\diagup \mathcal{Z}_3 \subset
\mathcal{B}_3\diagup \mathcal{Z}_3$ are represented by a word of form \eqref{eq10.15} and belong to the same conjugacy class, i.e. $\mathbold{b}_2=\mathbold{\beta} \mathbold{b}_1 \mathbold{\beta}^{-1}$ for an element $\mathbold {\beta} \in \mathcal{B}_3\diagup \mathcal{Z}_3 $. Then by the Lemmas
\ref{lemm1'} and  \ref{lem10.22} for the element $\mathbold{\beta}$ the equality
$\mathbold{\beta}=\mathbold{ \beta}_1 \mathbold{\Delta}_3^{\ell}$ holds 
with $\mathbold{\beta}_1$ being a word in 
$\sigma_1^2\diagup \mathcal{Z}_3 $ and $\sigma_2^2\diagup \mathcal{Z}_3 $,
$\mathbold{\Delta}_3=\Delta_3 \diagup \mathcal{Z}_3  $
and $\ell=0,1$.

Put $\mathbold{b}'_1= {\mathbold{\Delta}_3}^{\ell} \mathbold{b}_1{\mathbold{\Delta}_3}^{-\ell}$. 
Then $\mathbold{b}'_1$ is represented by a word of form \eqref{eq10.15c}. 
The element $\mathbold{b}_2$ is conjugate by an element of $\mathcal{PB}_3\diagup \mathcal{Z}_3$ either to  $\mathbold{b}_1$ or to  $\mathbold{b}'_1$. 
By Lemma \ref{lem10.20} 
the reduced word representing the element $\mathbold{b}_2$ is obtained form the reduced word representing either    $\mathbold{b}_1$ or to  $\mathbold{b}'_1$ 
by a cyclic permutation of terms. The number of cyclic permutations of $2j'_0$ letters, equals $2j'_0$. 
Hence the number of different elements of $\mathcal{PB}_3\diagup \mathcal{Z}_3$ that can be represented by words of form  \eqref{eq10.15}
and are obtained from $\mathbold{b}_1$ by conjugation with an element of $\mathcal{B}_3\diagup \mathcal{Z}_3$ does not exceed twice the number of cyclic permutations of $2j'_0$ letters, i.e. it does not exceed  $2\cdot2j'_0$. (Notice, that e.g. by cyclically permuting the terms of a word with symmetries we may sometimes arrive at the same word.)

We proved that the number of different conjugacy classes of $\mathcal{PB}_3\diagup \mathcal{Z}_3$ represented by words of form  \eqref{eq10.15}
with $2j'_0$ terms is not smaller than $\frac{2^{2j'_0}}{2\cdot 2j'_0} $.

We obtain
\begin{equation}\label{eq10.17}
N^{entr}_{\mathcal{B}_3}(Y) \geq \frac{2^{2j'_0-1}}{2\cdot 2j'_0}\,.
\end{equation}
Notice that $\frac{2^j}{j}\geq 2 $ for natural $j$. Indeed, the function $x \to \frac{2^x}{x}$ increases for $x \geq 2$ (since $ (\frac{2^x}{x})' = -\frac{2^x}{x^2}+\frac{2^x \log 2}{x}$ , and $\log 2 >0.6$) and for $j=1$ and $2$ the expression equals $2$).
Hence, $\frac{2^{2j}}{2j}\geq 2^{j}$. Hence, for $Y \geq \frac{70}{3}
\pi^2$
\begin{equation}\label{eq10.18}
N^{entr}_{\mathcal{B}_3}(Y) \geq  2^{j'_0-1}\geq \frac{1}{4}
2^{\frac{Y}{\frac{35}{6}\pi^2}}
\geq \frac{1}{4} e^{\frac{\log 2}{6\pi^2}\,Y}\,.
\end{equation}
The lower bound of the theorem is proved. \hfill $\Box$

\chapter{Riemann surfaces of second kind and finiteness theorems}
\label{chapterfin}
\chaptermark{Finiteness Theorems.}

\setcounter{equation}{0}

This chapter is our deepest application of the concept of conformal module and extremal length to
Gromov's Oka Principle. We address Problem \ref{problGrom2} that proposes to obtain information on the set of homotopy classes of continuous mappings from a connected open Riemann surface $X$ to a complex manifold $Y$. We consider the case when the target $Y$ is the twice punctured complex plane and
focus on the irreducible homotopy classes, since the reducible classes are well described and have the Gromov Oka property.
The information we are interested in is the number of irreducible homotopy classes of mappings $X\to Y$ that contain a holomorphic mapping.
More specifically, for any connected finite open Riemann surface $X$ (maybe, of second kind) we
give an effective upper bound for the number of irreducible holomorphic mappings up to homotopy from $X$ to the twice punctured complex plane, and an effective upper bound for the number of irreducible holomorphic torus bundles up to isotopy on such a Riemann surface.
The bound depends on a conformal invariant of the Riemann surface that is expressed in terms of the conformal module of a finite number of annuli in $X$.
On the other hand, in the proof we use instead of the conformal modules of conjugacy classes of elements of the fundamental group of  the twice punctured complex plane the more powerful extremal length with totally real boundary values of elements of the fundamental group themselves.

The estimates are in some sense asymptotically sharp:
If $X_{\sigma}$ is the $\sigma$-neighbourhood of a skeleton of a connected finite open hyperbolic Riemann surface with a Kähler metric,
then the number of irreducible holomorphic mappings up to homotopy from $X_{\sigma}$ to the twice punctured complex plane grows exponentially in $\frac{1}{\sigma}$.

These statements are analogs for Riemann surfaces of second kind of the Geometric Shafarevich Conjecture and the Theorem of de Franchis, that state the finiteness of the number of certain holomorphic objects on closed or punctured Riemann surfaces.

\section[Finiteness theorems.]{Riemann surfaces of first and of second kind and finiteness theorems. }\label{sec:fin.1}

It seems that the oldest finiteness theorem for mappings between complex manifolds is the following theorem, which was published by de Franchis \cite{Fr} in 1913. \index{de Franchis ! Theorem of}

\medskip

\noindent {\bf Theorem A (de Franchis).} {\it For closed connected Riemann surfaces $X$ and $Y$ with $Y$ of genus at least $2$ there are at most finitely many non-constant holomorphic mappings from $X$ to $Y$.}

\medskip

There is a more comprehensive Theorem in this spirit.

\medskip

\noindent {\bf Theorem B (de Franchis-Severi).} {\it For a closed connected Riemann surface $X$ there are
(up to isomorphism) only finitely many non-constant holomorphic mappings $f:X\to Y$ where $Y$ ranges over all closed Riemann surfaces of genus at least $2$.}
\index{de Franchis-Severi ! Theorem of}
\medskip

A finiteness theorem which became more famous because of its relation to number theory was conjectured by Shafarevich  \cite{Sh}.

\medskip

\noindent {\bf Theorem C (Geometric Shafarevich Conjecture.)} {\it For a given compact or punctured Riemann surface $X$ and given non-negative numbers $\textsf{g}$ and $\textsf{m}$ such that $2\textsf{g}-2+\textsf{m} >0$ there are only finitely many locally holomorphically non-trivial holomorphic fiber bundles over $X$ with fiber of type $(\textsf{g},\textsf{m})$.}
\index{Geometric Shafarevich Conjecture}
\smallskip

Theorem C was conjectured by Shafarevich \cite{Sh} in the case of compact base and fibers of type $(\textsf{g},0)$. It was
proved by Parshin \cite{Pa} in the case of compact base and fibers of type $(\textsf{g},0),\,\textsf{g}\geq 2,$ and by Arakelov \cite{Ara} for punctured Riemann surfaces as base and fibers of type $(\textsf{g},0)$. Imayoshi and Shiga \cite{IS} gave a proof of the quoted version using Teichm\"uller theory.
\index{Parshin} \index{Arakelov} \index{Imayoshi and Shiga} \index{Mordell} \index{McMullen} \index{Mazur}

The statement of Theorem C ''almost'' contains the so called Finiteness Theorem of Sections which is also called
the Geometric Mordell Conjecture (see \cite{Mc}),
giving an important conceptional connection between geometry and number theory.
For more details we refer to the surveys by C.McMullen \cite{Mc} and B.Mazur \cite{Maz}.

\medskip
Theorem A is a consequence of Theorem C, and Theorem A has analogs for the source $X$
and the target $Y$ being punctured Riemann surfaces. Indeed,
we may associate to any holomorphic mapping $f:X\to Y$ of Theorem A the bundle over $X$ with fiber over $x\in X$ equal to $Y$ with distinguished point $\{f(x)\}$. Thus, the fibers are of type $(\textsf{g},1)$.
A holomorphic self-isomorphism of a locally holomorphically non-trivial $(\textsf{g},1)$-bundle may lead to a new holomorphic mapping from $X$ to $Y$, but there are only finitely many different holomorphic self-isomorphisms.

We will consider here analogs of Theorems A and C for the case when the base $X$ is a Riemann surface of second kind. Notice that finite hyperbolic Riemann surfaces of second kind are interesting from the point of view of spectral theory of the Laplace operator with respect to the hyperbolic metric (see also \cite{Bo}). There are interesting relations to scattering theory and
(the Hausdoff dimension of) the limit set of the Fuchsian group defining $X$.

The Theorems A and C do not hold literally if the base $X$ is of second kind.
If the base is a Riemann surface of second kind the problem to be considered is the finiteness of the number of irreducible isotopy classes (homotopy classes, respectively) containing holomorphic objects. In case the base is a punctured Riemann surface this is equivalent to the finiteness of the number of holomorphic objects. For more detail see Sections \ref{sec:fin2} and \ref{sec:fin3}.

We will present here finiteness theorems with effective estimates for the case when the base is a Riemann surface of second kind.
The estimates depend on a conformal invariant of the Riemann surface.
We will now prepare the definition of this invariant.

Let $X$ be a connected open Riemann surface of genus $g\geq 0$ with $m+1$ holes, $m\geq 0$, equipped with a base point $q_0$. Recall that
the fundamental group $\pi_1(X,q_0)$ of $X$ is a free group in $2g+m$ generators.
We describe now the conformal invariant of the Riemann surface $X$ that will appear in the mentioned estimate.
We  fix a standard system of generators $\mathcal{E}$ of $\pi_1(X,q_0)$ that is associated to a standard bouquet of circles for $X$
(see Section \ref{sec:2.0}). The circles are denoted by $\alpha_j$, $\beta_j$ and $\gamma_j$ (see Figure \ref{fig8.3}).
The generators $\mathcal{E}\subset  \pi_1(X,q_0)$
are labeled as follows.
The elements $e_{2j-1,0}\in \pi_1(X,q_0),$  $j=1,\ldots ,g,$  are represented by the curves $\alpha_j$, the elements $e_{2j,0}\in \pi_1(X,q_0) , \, j=1,\ldots ,g,$  are represented by the curves $\beta_j$, and the elements $e_{2g+k,0}\in \pi_1(X,q_0), \, k=1,\ldots ,m,$  of $\pi_1(X,q_0)$ are represented by the curves $\gamma_k$.
\index{$\mathcal{E}$}

Let $\tilde X$ be the universal covering of $X$.
For each element $e_0 \in \pi_1(X,q_0)$ we consider the subgroup $\langle e_0\rangle$ of $\pi_1(X,q_0)$ generated by $e_0$.
Fix a point $\tilde{q}_0$ in $\tilde X$ that projects to the base point $q_0\in X$,
let  $({\rm Is}^{\tilde{q}_0})^{-1}(e_0)$
be the covering transformation corresponding to $e_0$ (see Section \ref{sec:2.0}), and $\langle ({\rm Is}^{\tilde{q}_0})^{-1}(e_0)\rangle$
the group generated by  $({\rm Is}^{\tilde{q}_0})^{-1}(e_0)$.
\begin{defn}\label{defnfin0}
Denote by $\mathcal{E}_j,\, j=2,\ldots,10,$ the set of primitive elements of $\pi_1(X,q_0)$ which can be written as product of at most $j$ factors with each factor being either an element of $\mathcal{E}$ or an element of $\mathcal{E}^{-1}$, the set of inverses of elements of $\mathcal{E}$. Define $\lambda_j=\lambda_j(X)$ as the maximum over $e_0 \in \mathcal{E}_j$ of the extremal length of the annulus $\tilde{X} \diagup \langle ({\rm Is}^{\tilde{q}_0})^{-1}(e_0)\rangle$.
\end{defn}
The quantity
$\lambda_7(X)$ (for mappings to the twice punctured complex plane), or $\lambda_{10}(X)$ (for $(1,1)$-bundles) is the mentioned conformal invariant.
\index{$\lambda_j(X)$}

\medskip

\index{mapping to a punctured sphere ! reducible}
In the following theorem we put $E=\{-1,1,\infty\}$. We will often refer to $\mathbb{P}^1\setminus \{-1,1,\infty\}$ as the thrice punctured Riemann sphere or the twice punctured complex plane $\mathbb{C}\setminus \{-1,1\}$. Recall that a continuous
mapping from a Riemann surface to the twice punctured complex plane is called  reducible, iff it is homotopic to a mapping with image in a once punctured disc contained in $\mathbb{P}^1\setminus E$. (The puncture may be equal to $\infty$.) See Definition \ref{defnGromov2}. For any connected finite open Riemann surface $X$ with only thick ends and non-trivial fundamental group
there are countably many non-homotopic reducible holomorphic mappings from $X$ to the twice punctured complex plane
(see Lemma \ref{lemGrom3}).
On the other hand the following theorem holds.
\begin{thm}\label{thmfin1}
For each open connected Riemann surface $X$ of genus $g\geq0$ with $m+1\geq 1$ holes there are up to homotopy at most
$3(\frac{3}{2}e^{36\pi \lambda_7(X)})^{2g+m}$
irreducible holomorphic mappings from $X$ into
$Y\stackrel{def}=\mathbb{P}^1\setminus \{-1,1,\infty\}$.
\end{thm}
Notice that the Riemann surface $X$ is allowed to be of second kind.
If $X$ is a torus with a hole, $\lambda_7(X)$ may be replaced by $\lambda_3(X)$.
If $X$ is a planar domain, $\lambda_7(X)$ may be replaced by $\lambda_4(X)$
\smallskip

Recall that a holomorphic $(1,0)$-bundle is also called a holomorphic torus bundle. A holomorphic torus bundle equipped with a holomorphic section is also considered as a holomorphic $(1,1)$-bundle.
Moreover, a smooth
$(1,0)$-bundle over a finite connected open smooth surface
admits a smooth section. A holomorphic torus bundle over a finite connected open Riemann surface
is (smoothly) isotopic to a holomorphic torus bundle that admits a holomorphic section.
The following theorem for torus bundles holds.
\begin{thm}\label{thmfin2}
Let $X$ be an open connected Riemann surface of genus $g\geq 0$ with $m+1\geq 1$ holes.
Up to isotopy
there are no more than
$\big(2 \cdot 15^6\cdot\exp(36 \pi \lambda_{10}(X))\big)^{2g+m}$
irreducible holomorphic $(1,1)$-bundles over $X$.
\end{thm}

For the definition of irreducible $(\sf g,\sf m)$-bundles
see Section \ref{sec:9.2}.
Since on each finite open Riemann surface with only thick ends and non-trivial fundamental group there are
countably many non-homotopic reducible holomorphic mappings with target being the twice punctured complex plane, there
are also countably many non-isotopic reducible holomorphic $(1,1)$-bundles over each such Riemann
surface. This follows from
Proposition \ref{propEl.2}.

Notice that Caporaso proved the existence of a uniform bound for the number of locally holomorphically non-trivial holomorphic fiber bundles
over closed Riemann surfaces of genus $g$ with $m\geq 0$ punctures
with fibers being closed Riemann surfaces of genus  $ \textsf{g}\geq 2$. The bound depends only on the numbers $g$, $\textsf{g}$ and $m$.  Heier gave effective uniform estimates, but the constants are huge and depend in a complicated way on the parameters.
\index{Caporaso} \index{Heier}

Theorems \ref{thmfin1} and \ref{thmfin2} imply effective estimates for the number of locally holomorphically non-trivial holomorphic $(1,1)$-bundles over punctured Riemann surfaces, however, the constants depend also on the conformal type of the base. More precisely, the following corollaries hold.
\begin{cor}\label{corfin1a}
There are no more than
$3(\frac{3}{2}e^{36 \pi \lambda_7(X)})^{2g+m}$
non-constant holomorphic mappings from a Riemann surface $X$ of type $(g,m+1)$ to $\mathbb{P}^1\setminus \{-1,1,\infty\}$.
\end{cor}

\begin{cor}\label{corfin1b}
There are no more than
$\big(2 \cdot 15^6\cdot\exp(36 \pi \lambda_{10}(X))\big)^{2g+m}$
locally holomorphically non-trivial holomorphic $(1,1)$-bundles over a Riemann surface $X$ of type $(g,m+1)$.
\end{cor}

The following examples demonstrate the different nature of the problem in the two cases, the case when the base is a punctured Riemann surface, and when it is a Riemann surface of second kind.

\medskip

\noindent {\bf Example 1.}
There are no non-constant holomorphic mappings from a torus with one puncture to the twice punctured complex plane. Indeed, by Picard's Theorem each such mapping extends to a meromorphic mapping from the closed torus to the Riemann sphere. This implies that the preimage of the set $\{-1,1,\infty\}$ under the extended mapping must contain at least three points, which is impossible.

The situation changes if $X$ is a torus with a large enough hole.
Let $\alpha\geq 1$ and $\sigma \in (0,1)$. Consider the torus with a hole $T^{\alpha,\sigma}$ that is obtained from $\mathbb{\mathbb{C}}\diagup (\mathbb{Z} + i \alpha \mathbb{Z}),$ (with $\alpha\geq 1$ being a real number) by removing a closed geometric rectangle of vertical side length $\alpha-\sigma$ and horizontal side length $1-\sigma$ (i.e. we remove a closed subset that lifts to such a closed rectangle in $\mathbb{C}$). A fundamental domain for
this Riemann surface is ''the golden cross on the Swedish flag'' turned by $\frac{\pi}{2}$ with width of the laths being $\sigma$ and length of the laths being $1$ and $\alpha$.

\begin{prop}\label{propfin1a}
Up to homotopy there are at most $7e^{2^5 \cdot 3^2 \pi \frac{2\alpha +1}{\sigma}}$
irreducible holomorphic mappings from  $T^{\alpha,\sigma}$
to the twice punctured complex plane.

On the other hand, there are positive constants $c$, $C$, and $\sigma_0$ such that for any positive number $\sigma<\sigma_0$ and any $\alpha \geq 1$ there are at least $ce^{C\frac{\alpha}{\sigma}}$ non-homotopic holomorphic mappings from $T^{\alpha,\sigma}$
to the twice punctured complex plane.

\end{prop}

\noindent {\bf Example 2.}
There are only finitely many holomorphic maps from a thrice punctured Riemann sphere to another thrice punctured Riemann sphere. Indeed, after normalizing both, the source and the target space,
by a M\"obius transformation we may assume that both
are equal to $\mathbb{C} \setminus \{-1,1\}$. Each holomorphic map from
$\mathbb{C} \setminus \{-1,1\}$ to itself extends to a meromorphic map from the Riemann sphere to itself, which maps the set $\{-1,1,\infty\}$ to itself and maps no other point to this set. By the Riemann-Hurwitz formula
the meromorphic map takes each value exactly once. Indeed, suppose it takes each value $l$ times for a natural number $l$. Then each point in $\{-1,1,\infty\}$ has ramification index $l$.
Apply the Riemann Hurwitz formula for the branched covering $ X=\mathbb{P}^1\to Y=\mathbb{P}^1$ of multiplicity $l$
$$
\chi(X)= l \cdot \chi(Y) -\sum_{x\in Y} (e_x-1).
$$
Here $e_x$ is the ramification index at the point $x$. For the Euler characteristic we have $\chi(\mathbb{P}^1)=2$, and $\sum_{x\in Y} (e_x-1)\geq \sum_{x=-1,1,\infty} (e_x-1)=3\,(l-1)$. We obtain $2\leq 2\,l-3\,(l-1)$ which is possible only if $l=1$.
We saw that each non-constant holomorphic mapping from $\mathbb{C}\setminus\{-1,1\}$ to itself extends to a conformal mapping from
the Riemann sphere to the Riemann sphere that maps the set $ \{-1,1,\infty\}$ to itself. There are only finitely may such maps, each a M\"obius transformation permuting the three points.
\index{Riemann-Hurwitz formula}

For Riemann surfaces of second kind
the situation changes, as demonstrated in the following proposition. The proposition does not only concern the
case when the Riemann surface equals $\mathbb{P}^1$ with three holes.
We consider an open Riemann surface $X$ of genus $g$ with $m\geq 1$ holes.

\begin{prop}\label{propfin1b} Let $X$ be a connected finite open hyperbolic Riemann surface of genus $g$ with $m+1$ holes, that is equipped with a K\"ahler metric. Suppose $S$ is a standard bouquet of piecewise smooth circles for $X$
with base point $q_0$.
We assume that $q_0$ is the only non-smooth point of the circles, and
all tangent rays to circles in $S$ at $q_0$ divide a disc in the tangent space into equal sectors.
Let $S_{\sigma}$ be the
$\sigma$-neighbourhood of $S$ (in the K\"ahler metric on $X$).

Then there exists a constant $\sigma_0>0$, and positive constants $C'$, $C''$,  $c'$, $c''$,
depending only on $X$, $S$ and the K\"ahler metric, such that for each positive $\sigma<\sigma_0$ the number $N_{S_{\sigma}}^{\mathbb{C}\setminus \{-1,1\}}$ of non-homotopic irreducible holomorphic mappings from $S_{\sigma}$ to the twice punctured complex plane satisfies the
inequalities
\begin{align}\label{eqfin4}
c'e^{\frac{c''}{\sigma}} \leq N_{S_{\sigma}}^{\mathbb{C}\setminus \{-1,1\}}\leq C'e^{\frac{C''}{\sigma}}\,.
\end{align}
\end{prop}
\index{Kähler metric}

The following proposition demonstrates the role of the conformal module (extremal length, respectively) with suitable horizontal boundary values for the homotopy problem of individual continuous mappings to holomorphic ones. The proposition concerns mappings
from small neighbourhoods of standard bouquets of circles for open Riemann surfaces and uses the estimates of the extremal lengths in terms of the $\mathcal{L}$-invariant of the words representing the monodromies along the standard generators of the fundamental group.

\begin{prop}\label{propfin5} Let $X$,  $q_0$, $S$ and $S_{\sigma}$ be as in Proposition $\ref{propfin1b}$. For each $j=1,\ldots2g+m$ we let $\gamma_j$ be the circle of the standard bouquet $S$ for $X$ that represents the standard generator $e_j$ of $\pi_1(X,q_0)$.
Denote by $\ell_j$ the length of $\gamma_j$ in the Kähler metric.
There exist positive constants $\sigma_0$ and $C_1$, depending only on $X$, $S$, $q_0$ and the Kähler metric, such that for each positive $\sigma<\sigma_0$
any continuous mapping  $f:S_{\sigma}\to \mathbb{C}\setminus \{-1,1\}$ with $f(q_0)=0$  is homotopic to a holomorphic mapping if the inequalities
\begin{align}\label{eqfin30}
\mathcal{L}(f_*(e_j))\leq C_1 \frac{\ell_j}{\sigma}
\end{align}
hold for all $j=1,\ldots2g+m$ for the induced mapping $f_*$ on fundamental groups with base points.

Vive versa, there exists a constant $C_2$ depending only on $X$, $S$, $q_0$ and the Kähler metric, such that for each holomorphic mapping $f:S_{\sigma}\to \mathbb{C}\setminus \{-1,1\}$ with $f(q_0)=0$ and each $j=1,\ldots2g+m$ the inequality
\begin{align}\label{eqfin31}
\mathcal{L}(f_*(e_j))\leq C_2 \frac{\ell_j}{\sigma}
\end{align}
holds for the induced mapping $f_*$ on fundamental groups with base points.
\end{prop}

Theorem \ref{thmfin1} and Propositions \ref{propfin1a}, \ref{propfin1b} and \ref{propfin5} contribute to Problem \ref{problGrom2} on Gromov's Oka Principle. Theorem \ref{thmfin1} gives an upper bound for the number of homotopy classes of holomorphic mappings from connected finite open Riemann surfaces  to the twice punctured complex plane. Theorem
\ref{thmfin2} gives an upper bound for the number of the isotopy classes of $(1,1)$-bundles
on connected finite open Riemann surfaces that contain holomorphic bundles.
Propositions \ref{propfin1a} and \ref{propfin1b} consider families of Riemann surfaces $Y_{\sigma},\,\sigma\in (0,\sigma_0) $, obtained by continuously changing the conformal structure of a fixed Riemann surface, and determine the growth rate for $\sigma\to 0$ of the number of irreducible holomorphic mappings $X_\sigma\to \mathbb{C}\setminus\{-1,1\}$ up to homotopy. In Proposition
\ref{propfin1a} the family of Riemann surfaces depends also on a second parameter $\alpha$, and the growth rate is determined in $\alpha$ and $\sigma$. The proof of the lower bound in both propositions uses solutions of a $\overline{\partial}$-problem. The solution in the case of Proposition \ref{propfin1a} uses a simple explicit formula.
Proposition \ref{propfin5} concerns the existence of homotopies of individual continuous mappings to holomorphic ones.

For the proof of the upper bound in the theorems and propositions we have to understand
what prevents a continuous map from a finite open Riemann surface  $X$ to $\mathbb{C}\setminus \{-1,1\}$  to be homotopic to a holomorphic map.
Proposition \ref{propfin2} below says that an irreducible map $X\to\mathbb{C}\setminus \{-1,1\}$ can only be homotopic to a holomorphic map, if the ''complexity'' of the monodromies of the map are compatible with conformal invariants of the source manifold.

\section[preliminaries on mappings, coverings, and extremal length] {Some further preliminaries on mappings, coverings, and extremal length}
\label{sec:fin1a}
In this section we will prepare the proofs of the Theorems.

\noindent {\bf Regular zero sets.}
We will call a subset of a smooth manifold $\mathcal{X}$ a simple relatively closed curve
if it is a connected component of a regular level set of a smooth real-valued function on $\mathcal{X}$.

Let $\mathcal{X}$ be a connected finite open Riemann surface. Suppose the zero set $L$ of a non-constant smooth real valued function on $\mathcal{X}$ is regular.
Each component of $L$ is either a simple closed curve or
it can be parameterized by a continuous mapping $\ell:(-\infty,\infty)\to \mathcal{X}$. We call a component of the latter kind a simple relatively closed arc in $\mathcal{X}$.

A relatively closed curve $\gamma$ in a connected finite open Riemann surface $\mathcal{X}$
is said to be contractible to a hole of $\mathcal{X}$, if the following condition holds. Consider $\mathcal{X}$ as domain $\mathcal{X}^c\setminus \cup\mathcal{C}_j$ on a closed Riemann surface $\mathcal{X}^c$. Here the $\mathcal{C}_j$ are the holes, each is either a closed topological disc  with smooth boundary or a point. The condition on $\gamma$ is the following. For each pair  $U_1$, $U_2$ of open subsets of $ \mathcal{X}^c$, $\cup \mathcal{C}_j\subset U_1\Subset U_2$, there exists a homotopy of $\gamma$ that fixes $\gamma\cap U_1$ and moves $\gamma$ into $U_2$.
Taking for $U_2$ small enough neighbourhoods of  $\cup \mathcal{C}_j$ we see that the homotopy moves $\gamma$ into an annulus adjacent to one of the holes.

For each relatively compact domain $\mathcal{X}'\Subset \mathcal{X}$ in $\mathcal{X}$ there is a finite cover of $L\cap \overline{\mathcal{X}'}$ by open subsets $U_k$ of $\mathcal{X}$ such that each $L\cap U_k$ is connected. Each set $L\cap U_k$ is contained in a component of $L$. Hence, only finitely many connected components of $L$ intersect $\overline{\mathcal{X}'}$.
Let $L_0$ be a connected component of $L$ which is a simple relatively closed arc parameterized by $\ell_0:\mathbb{R}\to \mathcal{X}$. Since each set $L_0\cap U_k$ is connected it is the image of an interval under $\ell_0$. Take real numbers $t_0^-$ and $t_0^+$ such that all these intervals are contained in $(t_0^-,t_0^+)$. Then the images $\ell\big((-\infty,t_0^-)\big)$ and
$\ell\big((t_0^+ ,+\infty)\big)$ are contained in $\mathcal{X}\setminus \overline{\mathcal{X}'}$, maybe, in different components. Such parameters $t_0^-$ and $t_0^+$ can be found for each relatively compact deformation retract $\mathcal{X}'$ of $\mathcal{X}$. Hence for each relatively closed arc $L_0\subset L$ the set of limit points $L_0^+$ of $\ell_0(t)$ for $t\to \infty$ is contained in a boundary component of $\mathcal{X}$. Also, the set of limit points $L_0^-$ of $\ell_0(t)$ for $t\to -\infty$ is contained in a boundary component of $\mathcal{X}$. The boundary components may be equal or different.

Moreover, if
$\mathcal{X}'\Subset \mathcal{X}$ is a relatively compact domain in $\mathcal{X}$ which is a deformation retract of $\mathcal{X}$, and a connected component $L_0$ of $L$ does not intersect $\overline{\mathcal{X}'}$ then $L_0$ is contractible to a hole of $\mathcal{X}$. Indeed, $\mathcal{X}\setminus \overline{\mathcal{X}'}$ is the union of disjoint annuli, each of which is adjacent to a boundary component of $\mathcal{X}$, and the connected set $L_0$ must be
contained in a single annulus.

Further, denote by $L'$ the union of all connected components of $L$ that are simple relatively closed arcs. Consider those components $L_j$ of $L'$ that intersect $\mathcal{X}'$. There are finitely many such $L_j$. Parameterize each $L_j$ by a mapping $\ell_j:\mathbb{R}\to \mathcal{X}$. For each $j$ we let $[t_j^-,t_j^+]$ be a compact interval for which
\begin{equation}\label{eqfin2d}
\ell_j(\mathbb{R}\setminus [t_j^-,t_j^+]) \subset \mathcal{X}\setminus \overline{\mathcal{X}'}\,.
\end{equation}
Let $\mathcal{X}''$, $\mathcal{X}'\Subset \mathcal{X}''\Subset \mathcal{X}$, be a domain which is a deformation retract of $\mathcal{X}$ such that  $\ell_j([t_j^-,t_j^+])\subset \mathcal{X}''$ for each $j$. Then all connected components of $L'\cap \mathcal{X}''$, that do not contain a set $\ell_j([t_j^-,t_j^+])$, are contractible to a hole of $\mathcal{X}''$. Indeed, each such component is contained in the union of annuli $\mathcal{X}''\setminus \overline{\mathcal{X}'}$.

\smallskip

\noindent {\bf Subgroups of covering transformations.}
Let again $q_0$ be the base point of a Riemann surface $X$, and let $\tilde{q}_0$ be the base point in the universal covering $\tilde X$ with ${\sf P}(\tilde{q}_0)=q_0$ for the covering map ${\sf P}: \tilde{ X}\to X$.
Let $N$ be a subgroup of $\pi_1(X,q_0)$, and $({\rm Is}^{\tilde{q}_0})^{-1}(N)$ the respective subgroup of covering transformations.. Denote by $X(N)$ the quotient $\tilde X \diagup ({\rm Is}^{\tilde{q}_0})^{-1}(N)$. We obtain a covering
$\omega^N_{{\rm Id}} :\tilde X \to X(N)$ with group of covering transformations isomorphic to $N$.
The fundamental group of $X(N)$ with base point $(q_0)_N\stackrel{def}= \omega_{\rm Id}^N(\tilde{q}_0)$ can be identified with $N$. \index{$X(N)$}

If $N_1$ and $N_2$ are subgroups of $\pi_1(X,q_0)$ and $N_1$ is a subgroup of $N_2$ (we write $N_1 \leq N_2$), then there is a covering map $\omega^{N_2}_{N_1} :\tilde X \diagup  ({\rm Is}^{\tilde{q}_0})^{-1}(N_1) \to \tilde X \diagup  ({\rm Is}^{\tilde{q}_0})^{-1}(N_2)$, such that $\omega^{N_2}_{N_1} \circ \omega^{N_1}_{\rm Id}=\omega^{N_2}_{\rm Id}$. Moreover,
the diagram Figure \ref{fig3} below is commutative. \index{$\omega^{N_2}_{N_1}$}

\begin{figure}[h]
\begin{center}
\includegraphics[width=6cm]{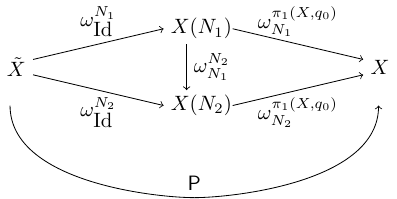}
\end{center}
\caption{A commutative diagram related to subgroups of the group of covering transformations}\label{fig3}
\end{figure}

\medskip

Indeed, take any point $x_1 \in \tilde X \diagup  ({\rm Is}^{\tilde{q}_0})^{-1}(N_1)$
and a preimage $\tilde x$ of $x_1$ under $\omega^{N_1}_ {\rm Id} $.
There exists
a neighbourhood $V(\tilde x)$ of $\tilde x$ in $\tilde X$ such that $V(\tilde x)\cap \sigma(V(\tilde x))=\emptyset$ for all covering transformations $\sigma\in {\rm Deck}(\tilde{X}, X)$.
Then for $j=1,2$ the mapping $\omega^{N_j,\tilde x} _{\rm Id} \stackrel{def}{=} \omega^{N_j}_{\rm Id}  \mid V(\tilde x)$ is a homeomorphism from $V(\tilde x)$ onto its image denoted by $V_j$. Put $x_2=\omega^{N_2,\tilde x} _{\rm Id}(\tilde{x})$. The set $V_j\subset \tilde{X}\diagup ({\rm Is}^{\tilde{q}_0})^{-1}(N_j)  $ is a neighbourhood of $x_j$ for $j=1,2$.

For each preimage $\tilde{x}' \in (\omega^{N_1}_{\rm Id})^{-1}(x_1)$
there is a covering transformation $\varphi_{\tilde x,\tilde{x}'}$ in $({\rm Is}^{\tilde{q}_0})^{-1}(N_1)$ which maps a neighbourhood $V(\tilde {x}')$ of $\tilde{x}'$ conformally onto the neighbourhood $V(\tilde {x})$ of $\tilde{x}$ so that
on $V(\tilde{x}')$ the equality
$\omega^{N_1,\tilde{x}'} _{\rm Id}  =  \omega^{N_1,\tilde x}_{\rm Id} \circ \varphi _{\tilde x,\tilde{x}'}$ holds. Choose $\tilde{x} \in (\omega^{N_1} _{\rm Id})^{-1}(x_1)$ and define
\begin{equation}\label{eqfin1a'}
\omega^{N_2}_{N_1}(y) = \omega^{N_2, \tilde x}_{N_1}(y) \stackrel{def}= \omega^{N_2, \tilde x} _{\rm Id} ((\omega^{N_1,\tilde x} _{\rm Id} )^{-1}(y)) {\mbox{ \;for each\;}} y \in V_1\,.
\end{equation}
We get a correctly defined mapping from $V_1$ onto $V_2$. Indeed, since $N_1$ is a subgroup of $N_2$, for another point $\tilde{x}' \in (\omega^{N_1}_ {\rm Id})^{-1}(x_1)$ the covering transformation $\varphi_{\tilde x,\tilde{x}'}$ is contained in $({\rm Is}^{\tilde{q}_0})^{-1}(N_2)$, and we get  the equality  $\omega^{N_2, \tilde{ x}'} _{\rm Id}  =  \omega^{N_2, \tilde x} _{\rm Id} \circ \varphi _{\tilde x,\tilde{x}'}$. Hence, for $y \in V_1(x_1)$
\begin{align}\label{eqfin1b'}
\omega^{N_2, \tilde{ x}'} _{\rm Id} \circ(\omega^{N_1,\tilde{x}'} _{\rm Id})^{-1} (y)  =& (\omega^{N_2, \tilde x} _{\rm Id} \circ \varphi_{\tilde x,\tilde{x}'})\circ (\omega^{N_1,\tilde x} _{\rm Id} \circ \varphi_{\tilde x,\tilde{x}'})^{-1}(y)\nonumber\\
=&
\omega^{N_2,\tilde x} _{\rm Id} \circ(\omega^{N_1,\tilde x} _{\rm Id} )^{-1}(y).
\end{align}
Since each mapping $\omega^{N_j,\tilde{x}}_ {\rm Id}$, $j=1,2,$ is a homeomorphism from
$V(\tilde{x})$ onto its image, the mapping $\omega^{N_2} _{N_1}$ is a homeomorphism
from $V(x_1)$ onto $V(x_2)$. The same holds for all preimages of
$V(x_2)$ under $\omega^{N_2} _{N_1}$. Hence, $\omega^{N_2} _{N_1}$
is a covering map. The commutativity of the part of the diagram
that involves the mappings $\omega^{N_1} _{\rm Id} $, $\omega^{N_2} _{\rm Id} $, and $\omega_{N_1}^{N_2}$ follows from equation \eqref{eqfin1a'}.

The existence of $\omega_{N_1}^{\pi_1(X,q_0)}$ and the equality ${\sf P}= \omega_{N_1}^{\pi_1(X,q_0)} \circ \omega^{N_1} _{\rm Id}$ follows by applying the above arguments with $N_2= \pi_1(X,q_0)$. The equality ${\sf P}= \omega_{N_2}^{\pi_1(X,q_0)} \circ \omega^{N_2} _{\rm Id}$ follows in the same way. Since
\begin{align}
{\sf P} =  &\omega_{N_2}^{\pi_1(X,q_0)} \circ \omega_{N_1}^{N_2}\circ  \omega^{N_1} _{\rm Id} \nonumber \\
=  &\omega_{N_1}^{\pi_1(X,q_0)}\quad\quad\;\;\,  \circ \omega^{N_1} _{\rm Id},\nonumber
\end{align}
we have
\begin{equation}\nonumber
\omega_{N_2}^{\pi_1(X,q_0)} \circ \omega_{N_1}^{N_2} = \omega_{N_1}^{\pi_1(X,q_0)}
\end{equation}

We will also use the notation  $\omega^N \stackrel{def}= \omega^N _{\rm Id}$ and $\omega_N \stackrel{def}= \omega_N ^{\pi_1(X,q_0)}$ for a subgroup $N$ of $\pi_1(X,q_0)$.

Let again $N_1 \leq N_2$ be subgroups of $\pi_1(X,q_0)$.
Consider the covering $\omega_{N_1}^{N_2} :\tilde X \diagup ({\rm Is}^{{\tilde q}_0})^{-1}(N_1) \to \tilde X \diagup ({\rm Is}^{{\tilde q}_0})^{-1}(N_2)$. Let $\beta$ be a simple relatively closed curve in  $\tilde X \diagup ({\rm Is}^{\tilde q_0})^{-1}(N_2)$.
Then $(\omega_{N_1}^{N_2})^{-1}(\beta)$ is the union of simple relatively closed curves in $\tilde X \diagup ({\rm Is}^{{\tilde q}_0})^{-1}( N_1)$ and $\omega_{N_1}^{N_2}: (\omega_{N_1}^{N_2})^{-1}(\beta) \to \beta$ is a covering.
Indeed, we cover $\beta$ by small discs $U_k$ in $\tilde X \diagup ({\rm Is}^{{\tilde q}_0})^{-1}( N_2)$ such that for each $k$ the restriction of $\omega_{N_1}^{N_2}$
to each connected component of $(\omega_{N_1}^{N_2})^{-1}(U_k)$ is a homeomorphism onto $U_k$, and $U_k$ intersects $\beta$
along a connected set. Take any $k$ with $U_k \cap \beta \neq \emptyset$. Consider the preimages $(\omega_{N_1}^{N_2})^{-1}(U_k)$.
Restrict  $\omega_{N_1}^{N_2}$  to the intersection of any preimage $(\omega_{N_1}^{N_2})^{-1}(U_k)$ with $(\omega_{N_1}^{N_2})^{-1}(\beta)$. We obtain a homeomorphism onto $U_k \cap \beta$.
It follows that the map $(\omega_{N_1}^{N_2})$ is a covering from each connected component of $(\omega_{N_1}^{N_2})^{-1}(\beta)$ onto $\beta$.

\noindent{\bf The extremal length of monodromies.}
Let as before $X$ be a connected finite open Riemann surface with base point $q_0$, and $\tilde{q}_0$ a point in the universal covering $\tilde X$ for which ${\sf P}(\tilde{q}_0)=q_0$ for the covering map ${\sf P}:\tilde{X}\to X$.

Recall that for an arbitrary point $q\in X$ the free homotopy class of an element $e$ of the fundamental group $\pi_1(X,q)$ can be identified with the conjugacy class of elements of $\pi_1(X,q)$ containing $e$ and is denoted by $\widehat{e}$. Notice that for $e_0\in \pi_1(X,q_0)$ and a curve $\alpha$ in $X$ with initial point $q_0$ and terminal point $q$ the free homotopy classes of $e_0$ and of $e={\rm Is}_{\alpha}(e_0)$ coincide, i.e. $\widehat{e}=\widehat{e}_0$.
Consider a simple smooth relatively
closed curve $L$ in $X$.  We will say that a free homotopy class of curves $\widehat {e_0}$ intersects $L$ if each representative of $\widehat{e_0} $ intersects $L$. Choose an orientation of $L$. The intersection number of $\widehat {e_0}$ with the oriented curve $L$ is the intersection number with $L$ of some (and, hence, of each) smooth loop representing $\widehat {e_0}$ that intersects $L$ transversally.
The intersection number of a loop with the relatively closed curve $L$ is the sum of the intersection numbers over all intersection points. The intersection number at an intersection point equals  $+ 1$ if the orientation determined by the tangent vector to the loop
followed by the tangent vector to $L$ is the orientation of $X$, and equals $-1$ otherwise.

Let $A$ be an annulus equipped with an orientation (called positive orientation) of simple closed dividing curves in $A$.
Recall that a relatively closed curve in a surface $X$ is called dividing, if $X\setminus \gamma$ consists of two connected components.
A continuous mapping $\omega:A\to X$ is said to represent a conjugacy class $\widehat{ e}$ of elements of the fundamental group $\pi_1(X,q)$ for a point $q\in X$,
if the composition
$\omega\circ\gamma$ represents $\widehat{ e}$ for each positively oriented dividing curve $\gamma$ in $A$ .

Let $A$ be an annulus with base point $p$ with a chosen positive orientation of simple closed dividing curves in $A$.
Let $\omega$ be a continuous mapping from $A$ to a finite Riemann surface ${X}$ with base point $q$ such that $\omega(p)=q$. We write $\omega:(A,p) \to ({X},q)$. The mapping is said to represent the element $e$ of the fundamental group $ \pi_1({X},q)$
if $\omega\circ \gamma$ represents $e$ for some (and hence for each) positively oriented simple closed dividing curve $\gamma$ in $A$ with base point $q$.

Let $q_0$ be the base point of $X$ chosen above and $\tilde{q}_0$ a point in the  universal covering $\tilde X$ with ${\sf P}(\tilde{q}_0)=q_0$.
We associate to each element $e_0 \in  \pi_1(X,q_0)$ of the free group $\pi_1(X,q_0)$ the annulus  $X(\langle e_0 \rangle)=\tilde{X}\diagup ({\rm Is}^{\tilde{q}_0})^{-1}(\langle e_0 \rangle)$ with base point $q_{\langle e_0\rangle}= \omega_{\rm Id}^{\langle e_0 \rangle}(\tilde{q}_0)$ and the covering map $\omega_{\langle e_0 \rangle} \stackrel{def}=\omega^{\pi_1(X,q_0)}_{\langle e_0 \rangle}:X(\langle e_0 \rangle) \to X$.
By the commutative diagram Figure \ref{fig3} the equality  $\omega_{\langle e_0 \rangle }(q_{\langle e_0 \rangle})= q_0$ holds.
We choose the orientation of simple closed dividing curves in $X(\langle e_0 \rangle)=\tilde{X}\diagup ({\rm Is}^{\tilde{q}_0})^{-1}(\langle e_0 \rangle)$
so that for a curve $\tilde \gamma$ in $\tilde X$ with initial point $\tilde{q}_0$ and terminal point ${\rm Is}^{\tilde{q}_0}(e_0)$ the curve
$\gamma_{\langle e_0 \rangle}\stackrel{def}=\omega^{\langle e_0 \rangle}(\tilde{ \gamma})$ is positively oriented.
The locally conformal mapping
$\omega^{\pi_1(X,q_0)}_{\langle e_0 \rangle}:(X(\langle e_0 \rangle), q_{\langle e_0\rangle})\to (X,q_0)$ represents $e_0$.
This follows from the equality $\omega_{\langle e_0 \rangle}(\gamma_{\langle e_0 \rangle })=\omega_{\langle e_0 \rangle}(\omega^{\langle e_0 \rangle}(\tilde{\gamma}))={\sf P}(\tilde{\gamma})=\gamma$, since ${\sf P}(\tilde{\gamma})$ represents $e_0$.

Take a curve $\alpha$ in $X$ that joins $q_0$ and $q$, and the point $\tilde{q}=\tilde{q}({\alpha})\in \tilde X$ such that $\alpha$ and $\tilde{q}$ are compatible, i.e. ${\rm Is}^{\tilde{q}}={\rm Is}_{\alpha}\circ {\rm Is}^{\tilde{q}_0}$ (see equation \eqref{eq1''}).
Put $e={\rm Is}_{\alpha}(e_0)$.
By equation \eqref{eq1''} $({\rm Is}^{\tilde{q}})^{-1}(e)=({\rm Is}^{\tilde{q}_0})^{-1}(e_0)$, hence,
$\tilde{X}\diagup ({\rm Is}^{\tilde{q}})^{-1}(\langle e\rangle)=\tilde{X}\diagup ({\rm Is}^{\tilde{q}_0})^{-1}(\langle e_0\rangle)=X(\langle e_0\rangle)$.
The locally conformal mapping $\omega_{\langle e_0\rangle}: X(\langle e_0\rangle)\to X$ takes the point $q_{\langle e\rangle}\stackrel{def}=\omega^{\langle e_0 \rangle}(\tilde{q})$ to $q\in X$. Moreover,
$\omega_{\langle e_0\rangle}: (X(\langle e_0\rangle),q_{\langle e\rangle})  \to (X,q)$ represents $e\in\pi_1(X,q)$. This can be seen by repeating the previous arguments.

Let $\alpha$ be an arbitrary curve in $X$ joining $q_0$ with $q$, and $\tilde{q}\in {\sf P}^{-1}(q)$ be arbitrary (i.e. $\alpha$ and $\tilde{q}$ are not required to be compatible).
Let $e\in \pi_1(X,q)$.
Denote the projection $\tilde{X}\to\tilde{X}\diagup ({\rm Is}^{\tilde{q}})^{-1}(\langle e\rangle)$ by $\omega^{\langle e\rangle,\tilde{q}}$, and the projection $\tilde{X}\diagup ({\rm Is}^{\tilde{q}})^{-1}(\langle e\rangle)\to X$ by $\omega_{\langle e\rangle,\tilde{q}}$. Put $q_{\langle e\rangle,\tilde{q}}=\omega^{\langle e\rangle,\tilde{q}}(\tilde{q})$.
For any such choice we choose the orientation of simple closed dividing curves on $\tilde{X}\diagup ({\rm Is}^{\tilde{q}})^{-1}(\langle e\rangle)$ so that $\omega^{\langle e\rangle,\tilde{q}}$ maps
any curve $\tilde{\gamma}$ in $\tilde X$ with initial point $\tilde{q}$  and terminal point $({\rm Is}^{\tilde{q}})^{-1}(\langle e\rangle)(\tilde{q})$  to a positively oriented dividing curve. We will call it the standard orientation of dividing curves in $\tilde{X}\diagup ({\rm Is}^{\tilde{q}})^{-1}(\langle e\rangle)$.
The mapping $\omega_{\langle e\rangle,\tilde{q}}: \big(\tilde{X}\diagup ({\rm Is}^{\tilde{q}})^{-1}(\langle e\rangle),q_{\langle e\rangle,\tilde{q}} \big)\to (X,q)$
represents $e$.

Since the mapping $\;\;\omega_{\langle e_0 \rangle }:(X(\langle e_0 \rangle),(q_0)_{\langle e_0\rangle}) \to (X,q_0)\;\;$ represents $\;e_0\;$,
the mapping $\omega_{\langle e_0 \rangle }:X(\langle e_0 \rangle) \to X$ represents the free homotopy class $\widehat{e_0}$. The following simple lemma will be useful.
\begin{lemm}\label{lemfin0'}
The annulus $X(\langle e_0 \rangle)$ has
smallest extremal length among annuli which admit a holomorphic mapping to $X$, that represents the conjugacy class $\widehat{e_0}$.
\end{lemm}
\noindent In other words, $X(\langle e_0 \rangle)$ is the "thickest" annulus with the property stated in Lemma \ref{lemfin0'}.

\medskip

\noindent {\bf Proof}. Take an annulus $A$ with a choice of positive orientation of simple closed dividing curves. Suppose $A \overset{\omega}{-\!\!\!\longrightarrow}X$
is a holomorphic mapping that represents $\widehat{e_0}$. The annulus
$A$ is conformally equivalent to a round annulus in the plane, hence, we may assume that $A$ has the form $A=\{z\in \mathbb{C}: \, r<|z|<R\}$ for  $0\leq r<R\leq \infty$ and the positive orientation of dividing curves is the counterclockwise one.

Take a positively oriented simple closed dividing curve $\gamma^{A}$ in $A$.
Its image $\omega\circ \gamma^{A}$ under $\omega$ represents the class $\widehat{e_0}$. Choose a point $q^{A}$ in $\gamma^{A}$, and put $q=\omega(q^{A})$.
Then $\gamma^{A}$ represents a generator of $\pi_1(A,q^{A})$ and $\gamma=\omega\circ\gamma^{A} $ represents an element $e$ of $\pi_1(X,q)$ in the conjugacy class $\widehat{e_0}$.
Choose a curve $\alpha$ in $X$ with initial point $q_0$ and terminal point $q$,
and a point $\tilde q$ in $\tilde X$ so that  $\alpha$ and $\tilde q$ are compatible,
and, hence, for $e={\rm Is}_{\alpha}(e_0)$ the equality $({\rm Is}^{\tilde{q}_0})^{-1}(e_0)=({\rm Is}^{\tilde{q}})^{-1}(e)$ holds.
Let $L$ be the relatively closed arc $\{q^A\cdot r: r\in \mathbb{R}\} \cap A$ in $A$ that contains $q^A$.
After a homotopy of $\gamma^{A}$ with fixed base point, we may assume that its base point $q^{A}$ is the only point of $\gamma^{A}$ that is contained in $L$.
The restriction $\omega|(A\setminus L)$ lifts to a mapping $\tilde{\omega}:(A\setminus L)\to \tilde X$, that extends continuously to the two strands $L_{\pm}$ of $L$. (We assume that $L_-$ is reached by moving clockwise from $A\setminus L$ towards $L$.)
Let $q^A_{\pm}$ be the copies of $q^A$ on the two strands  $L_{\pm}$. We choose the lift $\tilde{\omega}$ so that $\tilde{\omega}(q^A_-)=\tilde{ q}$.
Since the mapping $(A,q^{A})\to (X,q)$ represents $e$, we obtain $\tilde{\omega}(q^A_+)=\sigma(\tilde{ q})$ for $\sigma= ({\rm Is}^{\tilde{q}})^{-1}(e)$.
Then for each $z\in L$ the covering transformation $\sigma$ maps the point $\tilde{z}_-\in\tilde{\omega}(L_-)$ for which ${\sf P}(\tilde{z}_-)=z$ to the point $\tilde{z}_+\in\tilde{\omega}(L_+)$ for which ${\sf P}(\tilde{z}_+)=z$. Hence $\omega$ lifts to
a holomorphic mapping $\iota: A\to X(\langle e_0 \rangle)$.
By Lemma \ref{lemm6}  (see also Lemma  7 of \cite{Jo3}) $\lambda(A) \geq \lambda( X(\langle e_0 \rangle))$. \hfill $\Box$

\medskip

For each point $q\in X$ and each element $e\in \pi_1(X,q)$ we denote by $A(\widehat {e})$ the conformal class of the ''thickest'' annulus that admits a holomorphic mapping into $X$ that represents $\widehat{e}$. We saw that $\lambda(A(\widehat{e_0}))=\lambda(\tilde{X}\diagup ({\rm Is}^{\tilde{q}_0})^{-1}(\langle e_0 \rangle))$ for $e_0\in \pi_1(X,q_0)$.
By the same reasoning as before $\lambda(A(\widehat{e}))=\lambda(\tilde{X}\diagup ({\rm Is}^{\tilde{q}'})^{-1}(\langle e \rangle))$ for each
$\tilde{q}' \in \tilde X$ and each element $e \in \pi_1(X,{\sf P}(\tilde{q}'))$.
Hence, if $e_0$ and $e$ are free homotopic (equivalently, if they represent to the same conjugacy class,) then $\lambda(\tilde{X}\diagup ({\rm Is}^{\tilde{q}_0})^{-1}(\langle e_0 \rangle))=\lambda(\tilde{X}\diagup ({\rm Is}^{\tilde{q}})^{-1}(\langle e \rangle))$ for any $\tilde{q}_0\in {\sf P}^{-1}(q_0)\subset \tilde X$ and any $\tilde{q}\in {\sf P}^{-1}(q)$.
Notice that $A(\reallywidehat{e^{-1}})=A(\widehat e)$ for each $e\in \pi_1(X,q),\, q\in X$.

\section{Holomorphic mappings into the twice punctured plane}
\label{sec:fin2}

The following lemma will be crucial for the estimate of the $\mathcal{L}_-$-invariant of the monodromies of holomorphic mappings from a finite open Riemann surface to $\mathbb{C}\setminus \{-1,1\}$. For a point $q'\in (-1,1)$ we will identify $\pi_1( \mathbb{C}\setminus \{-1,1\},q')$ with  $\pi_1( \mathbb{C}\setminus \{-1,1\},0)$ by the canonical isomorphism ${\rm Is}_{\alpha_{q'}}$ for the curve $\alpha_{q'}$ that runs along the line segment in $(-1,1)$ joining $0$ with $q'$.
\begin{lemm}\label{lemfin2}
Let $f:X \to \mathbb{C}\setminus \{-1,1\}$ be a non-contractible holomorphic mapping
on a connected finite open Riemann surface $X$, such that $0$ is a regular value of ${\rm Im}f$. Assume that $L_0$ is a simple relatively closed curve in $X$ such that $f(L_0)\subset (-1,1)$. Let $q \in L_0$ and
$q'=f(q)$.

If for an element $e\in \pi_1(X,q)$
the free homotopy class $\widehat e$ intersects $L_0$,
then either the reduced word $f_*(e) \in \pi_1(\mathbb{C}\setminus \{-1,1\},q')$ is a  non-zero power of a standard generator of $\pi_1(\mathbb{C}\setminus \{-1,1\}, q')$ or the inequality
\begin{equation}\label{eqfin1a}
\mathcal{L}_-(f_*(e)) \leq 2\pi \lambda(A(\widehat{e}))
\end{equation}
holds.
\end{lemm}
Notice that we make a normalization in the statement of the Lemma by  requiring that $f$ maps $L_0$ into the interval $(-1,1)$, not merely into $\mathbb{R}\setminus \{-1,1\}$.

Lemma \ref{lemfin2}
will be a consequence of the following lemma.

\begin{lemm}\label{lemfin1}
Let $X$, $f$, $L_0$, $q\in L_0$ be as in Lemma $\ref{lemfin2}$, and $e\in \pi_1(X,q)$.
Let $\tilde q$ be an arbitrary point in ${\sf P}^{-1}(q)$.
Consider the annulus $A\stackrel{def}=\tilde{X}\diagup ({\rm Is}^{\tilde{q}})^{-1}(\langle e \rangle)$ and the holomorphic projection $\omega_{A}\stackrel{def}=\omega_{\langle e\rangle,\tilde{q}}$.
Put $q_{A}\stackrel{def}=\omega^{\langle e\rangle,\tilde{q}}(\tilde{q})$. The mapping $\omega_{A}:(A,q_{A})\to (X,q)$ represents $e$.

Let $L_{A}\subset A$ be the connected component of $(\omega_{A})^{-1}(L_0)$ that contains $q_{A}$.
If $\widehat e$ intersects $L_0$, then $L_{A}$ is a relatively closed curve in $A$ that has limit points on both boundary components of $A$, and the lift $f\circ \omega_{A}$ is a holomorphic mapping from $A$ to $\mathbb{C}\setminus\{-1,1\}$ that maps $L_{A}$ into $(-1,1)$.
\end{lemm}

\noindent{\bf Proof of Lemma \ref{lemfin1}.}
Let $\gamma:[0,1]\to X$ be a curve with base point $q$ in $X$ that represents $e$, and let $\tilde{\gamma}$ be the lift of $\gamma$ to $\tilde X$ with initial point $\tilde{\gamma}(0)$ equal to $\tilde q$.
Put $\sigma\stackrel{def}= ({\rm Is}^{\tilde{q}})^{-1}(e)$.
Then the terminal point
$\tilde{\gamma}(1)$ equals $\sigma(\tilde{q})$.

All connected components of ${\sf P}^{-1}(L_0)$ are relatively closed curves in $\tilde{X}\cong \mathbb{C}_+$ (where $\mathbb{C}_+$ denotes the upper half-plane) with limit points on the boundary of $\tilde X$. Indeed, the lift $f\circ {\sf P}$ of $f$ to $\tilde X$ takes values in $(-1,1)$ on ${\sf P}^{-1}(L_0)$. Hence, $|\exp(\pm \,i\, f\circ {\sf P})|=1$ on ${\sf P}^{-1}(L_0)$. A compact connected component of ${\sf P}^{-1}(L_0)$ would bound a relatively compact topological disc in $\tilde{X}=\mathbb{C}_+$, and by the maximum principle $|\exp(\pm \,i\, f\circ {\sf P})|=1$ on the disc. This would imply that $ f\circ {\sf P}$ is constant on $\tilde X$ in contrary to the assumptions.

Let $\tilde{L}_{\tilde{q}}$ be the connected component of ${\sf P}^{-1}(L_0)$ that contains $\tilde{q}$.  The point $\sigma(\tilde{q})$ cannot be contained in $\tilde{L}_{\tilde{q}}$. Indeed, assume the contrary. Then the arc  $\tilde{\alpha}$
on  $\tilde{L}_{\tilde{q}}$ that joins $\tilde{q}$ and $\sigma(\tilde{q})$  is homotopic in $\tilde X$ with fixed endpoints to $\tilde{\gamma}$.
The projection $\alpha={\sf P}(\tilde{\alpha})$
is contained in $L_0$ and is homotopic in $X$ with fixed endpoints to $\gamma$. Since $\gamma$ represents $e$ and $e$ is a primitive element of the fundamental group $\pi_1(X,q)$, this is possible only if $L_0$ is compact (and after orienting it) $L_0$ represents $e$.
A small translation of $\alpha$ to a side of $L_0$ gives a curve in $X$ that does not intersect $L_0$ and represents the free homotopy class $\widehat e$ of $e$.
This contradicts the fact that $\widehat{e}$ intersects $L_0$.

Take any other point $\tilde{q}'\in \tilde{L}_{\tilde{q}}$. If $\sigma(\tilde{q}')$ was in $ \tilde{L}_{\tilde{q}}$, then the arc on  $\tilde{L}_{\tilde{q}}$ that joins $\tilde q$ with  $\tilde{q}'$ would project to a loop that represents $\hat e$ and is contained in $L_0$. By the arguments above this is not possible.
We proved that the curves $\tilde{L}_{\tilde{q}}$ and $\sigma(\tilde{L}_{\tilde{q}})$ are disjoint.

Each of the two connected components $\tilde{L}_{\tilde{q}}$ and $\sigma(\tilde{L}_{\tilde{q}})$ divides $\tilde X$. Let $\Omega$ be the domain on $\tilde X$ that is bounded by $\tilde{L}_{\tilde{q}}$ and $\sigma(\tilde{L}_{\tilde{q}})$ and parts of the boundary of $\tilde X$. After a homotopy of $\tilde{\gamma}$  that fixes the endpoints we may assume that $\tilde{\gamma}((0,1))$ is contained in $\Omega$. Indeed, for each connected component of  $\tilde{\gamma}((0,1))\setminus \Omega$ there is a homotopy with fixed endpoints that moves the connected component to an arc on $\tilde{L}_{\tilde{q}}$ or  $\sigma(\tilde{L}_{\tilde{q}})$. A small perturbation yields a curve $\tilde{\gamma}'$ which is homotopic with fixed endpoints to $\tilde{\gamma}$ and has interior contained in $\Omega$. Notice that by the same reasoning as above, $\tilde{\gamma}'((0,1))$ does not meet any $\sigma^k(\tilde{L}_{\tilde{q}})$.

The curve $\omega^{\langle e\rangle,\tilde{q}}(\tilde{\gamma}')$ is a closed curve on $A$ that represents a generator of the fundamental group of $A$ with base point $q_A$. Moreover, $\omega_A\circ \omega^{\langle e\rangle,\tilde{q}}(\tilde{\gamma}')=  \omega_{\langle e\rangle,\tilde{q}}\circ \omega^{\langle e\rangle,\tilde{q}}(\tilde{\gamma}')={\sf P}(\tilde{\gamma}')$ represents $e$. Hence, the mapping $\omega_A:(A,q_A)\to (X,q)$ represents $e$.

The curve $\omega^{\langle e\rangle,\tilde{q}}(\tilde{\gamma}')$ intersects $L_{\langle e\rangle}=\omega^{\langle e\rangle,\tilde{q}}(\tilde{L}_{\tilde{q}})$ exactly once.
Hence, $L_{\langle e\rangle}$ has limit points on both boundary circles of $A$ for otherwise $L_{\langle e\rangle}$ would intersect one of the components of $A\setminus \omega^{\langle e\rangle,\tilde{q}}(\tilde{\gamma}')$ along a set which is relatively compact in $A$, and $\tilde{\gamma}'$ would have intersection number zero with         $L_{\langle e\rangle}$. It is clear that $f\circ \omega_A(L_A)=f(L_0)\subset (-1,1)$.
The lemma is proved. \hfill $\Box$

\medskip

\noindent {\bf Proof of Lemma \ref{lemfin2}.}
Let $\omega_A:(A,q_A)\to (X,q)$ be the holomorphic mapping from Lemma \ref{lemfin1} that represents $e$, and let $L_A\ni q_A$ be the relatively closed curve in $A$ with limit set on both boundary components of $A$. Consider a positively oriented dividing curve  ${\gamma}_A:[0,1]\to A$ with base point $\gamma(0)=\gamma(1)=q_A$ such that $\gamma_A((0,1))\subset A\setminus L_A$. The curve $\gamma=\omega_A(\gamma_A)$ represents $e$. The mapping $f\circ \omega_A$ is holomorphic on $A$ and $f\circ \omega_A(\gamma_A)=f(\gamma)$ represents $f_*(e)\in \pi_1(\mathbb{C}\setminus\{-1,1\}, q')$ with $q'=f\circ \omega_A(q_A)=f(q)\in (-1,1)$. Hence,
$f\circ \omega_A(\gamma_A)$ also represents the element $(f_*(e))_{tr}\in \pi_1^{tr}(\mathbb{C}\setminus\{-1,1\})$ in the relative fundamental group  $\pi_1(\mathbb{C}\setminus \{-1,1\},(-1,1))=\pi_1^{tr}(\mathbb{C}\setminus\{-1,1\})$ corresponding to $f_*(e)$.

We prove now that $\Lambda_{tr}(f_*(e))\leq \lambda(A)$.
Let $A_0\Subset A$ be any relatively compact annulus in $A$ with smooth boundary such that $q_A\in A_0$. If $A_0$ is sufficiently large, then the connected component $L_{A_0}$ of $L_A\cap A_0$ that contains
$q_A$ has endpoints on different boundary components of $A_0$. The set $A_0 \setminus L_{A_0}$ is a curvilinear rectangle.
The open horizontal curvilinear sides are the strands of the cut that are reachable from the curvilinear rectangle moving counterclockwise, or clockwise, respectively. The open vertical curvilinear sides are obtained from the boundary circles of $A_0$ by removing an endpoint of the arc $L_{A_0}$. Since $f\circ\omega_A$ maps $L_{A}$ to $(-1,1)$, the restriction of  $f\circ\omega_A$ to $A_0 \setminus L_{A_0}$ represents $(f_*(e))_{tr}$. Hence,
\begin{equation}\label{eqfin2a}
\Lambda_{tr}(f_*(e))\leq \lambda(A_0 \setminus L_{A_0})\,.
\end{equation}
By Corollary \ref{cor4.0}
\begin{equation}\label{eqfin2b}
\lambda(A_0 \setminus L_{A_0})\leq \lambda(A_0)\,.
\end{equation}
We obtain the inequality
$\Lambda_{tr}(f_*(e))   \leq \lambda(A_0)\,$
for each annulus $A_0\Subset A$,
hence, since $A$ belongs to the class $A(\widehat{e})$ of conformally equivalent annuli,
\begin{align}\label{eqfin2c}
\Lambda_{tr}(f_*(e))   \leq \lambda(A(\widehat{e}))\,,
\end{align}
and the Lemma follows from Theorem \ref{thm1}.
\hfill $\Box$

\medskip

\noindent {\bf The monodromies along two generators.}
In the following Lemma we combine the information on the monodromies along two generators of the fundamental group $\pi_1(X,q)$. We allow the situation when the monodromy along one generator or along each of the two generators of the fundamental group of $X$ is a power of a standard generator of $\pi_1(\mathbb{C}\setminus \{-1,1\},f(q))$.

\begin{lemm}\label{lemfin3}
Let $f:X \to \mathbb{C} \setminus \{-1,1\}$ be a non-contractible holomorphic function on a connected open Riemann surface $X$ such that $0$ is a regular value of the imaginary part of $f$.
Suppose $f$ maps a simple relatively closed curve $L_0$ in $X$ to $(-1,1)$, and $q$ is
a point in $L_0$. Let $e^{(1)}$ and $e^{(2)}$ be primitive elements of $\pi_1(X,q)$. Suppose that for each $e= e^{(1)},\;e= e^{(2)}$, and $e=e^{(1)}\,e^{(2)}$,
the free homotopy class $\widehat e$ intersects $L_0$.
Then either $f_*(e^{(j)}),\, j=1,2,\,$ are (trivial or non-trivial) powers of the same standard generator of $\pi_1(\mathbb{C} \setminus \{-1,1\},q')$ with $q'=f(q) \in (-1,1)$, or each of them is the product of at most two elements $w_1$ and $w_2$ of $\pi_1(\mathbb{C} \setminus \{-1,1\},q')$ with
\begin{equation}\nonumber
\mathcal{L}_-(w_j) \leq  2\pi \lambda_{e^{(1)},e^{(2)}},\, j=1,2,
\end{equation}
where
\begin{equation}\nonumber
\lambda_{e^{(1)},e^{(2)}} \stackrel{def}=\max\{\lambda(A(\reallywidehat{e^{(1)}})),\, \lambda(A(\reallywidehat{e^{(2)}})),\, \lambda(A(\reallywidehat{e^{(1)}\,e^{(2)}}))\}.
\end{equation}
Hence,
\begin{equation}\label{eqfin3e'}
\mathcal{L}_-(f_*(e^{(j)})) \leq 4 \pi \lambda_{e^{(1)},e^{(2)}}, \, j=1,2.
\end{equation}
\end{lemm}
\index{$\lambda_{e^{(1)},e^{(2)}}$}

\noindent {\bf Proof.}
If the monodromies $f_*(e^{(1)})$ and $f_*(e^{(2)})$ are not powers of a single standard generator (the identity is considered as zeroth power of a standard generator) we obtain the following.
At most two of the elements, $f_*(e^{(1)})$,  $f_*(e^{(2)})$, and $f_*(e^{(1)}\, e^{(2)})= f_*(e^{(1)})\, f_*(e^{(2)})$, are powers of a standard generator, and if two of them are powers of a standard generator, then they are non-zero powers of different standard generators.

If two of them are non-zero powers of different standard generators, then the third has the form
$a_{\ell }^{k} a_{\ell '}^{k'}$ with $a_{\ell }$ and $a_{\ell ' }$
being different generators and $k$ and $k'$ being non-zero integers.
By Lemma \ref{lemfin2} the $\mathcal{L}_-$ of the third element
does not exceed $2\pi \lambda_{e^{(1)},e^{(2)}}$. On the other hand it equals   $\log(3|k'|)+ \log(3|k|)$. Hence, $\mathcal{L}_-(a_{\ell }^{k}) = \log(3|k|)\leq 2\pi \lambda_{e^{(1)},e^{(2)}}$
and  $\mathcal{L}_-(a_{\ell' }^{k'}) = \log(3|k'|)\leq 2\pi \lambda_{e^{(1)},e^{(2)}}$.

If two of the elements $f_*(e^{(1)})$,  $f_*(e^{(2)})$, and $f_*(e^{(1)}\, e^{(2)})= f_*(e^{(1)})\, f_*(e^{(2)})$, are not powers of a standard generator, then the $\mathcal{L}_-$ of each of the two elements does not exceed $2\pi \lambda_{e^{(1)},e^{(2)}}$. Since the $\mathcal{L}_-$ of an element coincides with the $\mathcal{L}_-$ of its inverse, the third element is the product of two elements with $\mathcal{L}_-$ not exceeding $2\pi \lambda_{e^{(1)},e^{(2)}}$.
Since for $x,x'\geq 2$ the inequality $\log(x+x')\leq \log x +\log x'$ holds,
the $\mathcal{L}_-$ of the product does not exceed the sum of the $\mathcal{L}_-$ of the factors. Hence the $\mathcal{L}_-$ of the third element does not exceed $4\pi \lambda_{e^{(1)},e^{(2)}}$. Hence, inequality \eqref{eqfin3e'} holds.
\hfill $\Box$

\medskip

The following proposition states
the existence of suitable connected components of the zero set of the imaginary part of certain analytic functions on tori with a hole and on planar domains.
As in Section  \ref{sec:fin.1} for a Riemann surface $X$ we let $\mathcal{E}$ be a standard system of generators of $ \pi_1(X,q_0)$
that is associated to a standard bouquet of circles for $X$.
Recall that $\mathcal{E}_j$ is the set of primitive elements of $ \pi_1(X,q_0)$ which can be written as product of at most $j$ elements of $\mathcal{E}\cup (\mathcal{E})^{-1} $.
Here as before for any subset $\mathcal{E}'$ of $\pi_1(X;q_0)$ we denote by $(\mathcal{E}')^{-1}$ the set of all elements that are inverse to elements in  $\mathcal{E}'$.
\index{$\mathcal{E}^{-1}$}

\begin{prop}\label{propfin2a} Let $X$ be a torus with a hole or a planar domain with base point $q_0$ and fundamental group $\pi_1(X,q_0)$, and let $\mathcal{E}$  be a standard system of generators  of $\pi_1(X,q_0)$ that is associated to a standard bouquet of circles for $X$.
Let $f: X \to \mathbb{C}\setminus \{-1,1\}$ be a non-contractible holomorphic mapping
such that $0$ is a regular value of ${\rm{Im}} f$. Then there exist a simple relatively closed curve $L_0\subset X$ such that
$f(L_0) \subset \mathbb{R}\setminus \{-1,1\},$ and a set $\mathcal{E}_2'  \subset \mathcal{E}_2 \subset\pi_1(X,q_0)$ of primitive elements of $\pi_1(X,q_0)$,
such that the following holds.
Each element $e_{j,0} \in \mathcal{E} \subset \pi_1(X,q_0)$ is the product of at most two elements of $\mathcal{E}_2'\cup (\mathcal{E}_2')^{-1} $.  Moreover, for each $e_0 \in \pi_1(X,q_0)$ which is the product of one or two elements from $\mathcal{E}_2'$ the free homotopy class $\widehat{ e_0}$ has positive intersection number with $L_0$ (after suitable orientation of $L_0$).

If $X$ is a torus with a hole or $X$ equals $\mathbb{P}^1$ with three holes, we may chose $\mathcal{E}_2'$ consisting of two elements, one of them contained in $\mathcal{E}$, the other is either contained in $\mathcal{E}\cup \mathcal{E}^{-1}$ or is a product of two elements of $\mathcal{E}$.
\end{prop}
Notice the following facts. By Theorem \ref{thmEl-1} a mapping $f:X\to \mathbb{C}\setminus \{-1,1\}$ is contractible if and only if  for each $e_0\in \pi_1(X,q_0)$ the monodromy $f_*(e_0)$ is equal to the identity. The mapping $f$ is reducible if and only if the monodromy mapping $f_*:\pi_1(X,q_0)\to \pi_1(\mathbb{C}\setminus\{-1,1\},f(q_0))$ is conjugate to a mapping into a subgroup $\Gamma$ of $\pi_1(\mathbb{C}\setminus\{-1,1\},f(q_0))$ that is generated by a single element that is represented by a curve which separates one of the points $1,-1$ or $\infty$ from the other points. In other words, $\Gamma$ is (after identifying fundamental groups with different base point up to conjugacy) generated by a conjugate of one of the elements $a_1$, $a_2$ or $a_1a_2$ of $\pi_1(\mathbb{C}\setminus\{-1,1\},0)$. Recall that for $q'\in (-1,1)$ the fundamental group $\pi_1(\mathbb{C}\setminus\{-1,1\},q')$ is canonical isomorphic to $\pi_1(\mathbb{C}\setminus\{-1,1\},0)$ (by changing the base point along the segment in $(-1,1)$ that joins $q'$ and $0$), and we will identify
$\pi_1(\mathbb{C}\setminus\{-1,1\},q')$ with  $\pi_1(\mathbb{C}\setminus\{-1,1\},0)$.

If $f$ is irreducible, then it is not contractible, and, hence, the preimage $f^{-1}(\mathbb{R})$ is not empty.

Denote by $M_1$ a M\"obius transformation which permutes the points  $-1,\,1,\, \infty$ and maps the interval $(-\infty,-1)$ onto $(-1,1)$, and let $M_2$ be a M\"obius transformation which permutes the points $-1,\,1,\, \infty$ and maps the interval $(1,\infty)$ onto $(-1,1)$. Let $M_0\stackrel{def}= \mbox {Id}$.

The main step for the proof of Theorem \ref{thmfin1} is the following Proposition \ref{propfin2}.

Recall that the invariant $\lambda_j(X)$ was defined in Definition \ref{defnfin0}, Section \ref{sec:fin.1}. Since for $e_0\in \pi_1(X,q_0)$
the equality $\lambda(\tilde{X}\diagup ({\rm Is}^{\tilde{q}_0})^{-1}(\langle e_0\rangle))=\lambda(A(\widehat{e_0}))$ holds, $\lambda_j(X)$ is the maximum of $\lambda(A(\widehat{e_0}))$ over $e_0\in \mathcal{E}_j$.

\begin{prop}\label{propfin2} Let $X$ be a connected finite open Riemann surface with base point $q_0$, and let $\mathcal{E}$ be the same system of generators of $\pi_1(X,q_0)$ as in Proposition $\ref{propfin2a}$.
Suppose
$f:X \to \mathbb{C}\setminus \{-1,1\}$ is an irreducible holomorphic mapping, such that
$0$ is a regular value of $\,{\rm{Im}}f$. Then for one of the functions $M_l \circ f,\, l=0,1,2,\,$
which we denote by $F$, there exists a point $q \in X$ (depending on $f$), such that
the point
$q' \stackrel{def}=F(q)$ is contained in  $(-1,1)$, and there exists a smooth curve $\alpha$ in $X$ joining $q_0$ with $q$, such that the following holds.
For each element $e_j\in\mbox{Is}_{\alpha}(\mathcal{E})$
the monodromy $F_*(e_j)$ is the product of at most six elements of $\pi_1(\mathbb{C}\setminus\{-1,1\},q')$ (canonically identified with   $\pi_1(\mathbb{C}\setminus\{-1,1\},0)$)
of $\mathcal{L}_-$ not exceeding $2 \pi \lambda_7(X)$ and, hence,
\begin{equation}\label{eqfin3b}
\mathcal{L}_-(F_*(e_j)) \leq 12\pi \lambda_7(X) \;\; \mbox{for each}\;\, j.
\end{equation}
If $X$ is a torus with a hole the proposition holds with $\lambda_7(X)$ replaced by $\lambda_3(X)$. If $X$ is a planar domain the proposition holds with $\lambda_4(X)$ instead of $\lambda_7(X)$.
\end{prop}
Notice, that
all monodromies of contractible mappings are equal to the identity, hence the inequality \eqref{eqfin3b} holds automatically for contractible mappings.

\noindent We postpone the proof of the two propositions and prove first the Theorem \ref{thmfin1}.

\smallskip

\noindent {\bf Proof of Theorem \ref{thmfin1}.}
Let $X$ be a connected finite open Riemann surface (possibly of second kind)
with base point $q_0$. We need to estimate the number $N_X$ of homotopy classes of irreducible mappings $X\to \mathbb{C}\setminus \{-1,1\}$ that contain a holomorphic mapping.
Take  any relatively compact open subset $X^0\Subset X$ that is a deformation retract of $X$ and contains $q_0$. We claim that the number $N_X$ does not exceed the number  $N_{X^0}^*$  of homotopy classes of irreducible mappings $X^0\to \mathbb{C}\setminus \{-1,1\}$ that contain a
holomorphic mapping for which $0$ is a regular value of the imaginary part.

Notice first that homotopic mappings $X\to \mathbb{C}\setminus \{-1,1\}$ restrict to  homotopic mappings $X^0\to \mathbb{C}\setminus \{-1,1\}$. Vice versa, if the restrictions $f_j\mid X^0$ of two mappings $f_j: X\to \mathbb{C}\setminus \{-1,1\}$, $j=1,2,$ are homotopic then
the $f_j$ are homotopic. Indeed, let $\varphi_t:X\to X$ be a continuous family of mappings $\varphi_t:X\to X,\, t\in [0,1],$ that map $X$ homeomorphically onto a domain $X_t$ in $X$ such that $\varphi_0$ is the identity and $X_1=X^0$ (see also Lemma \ref{lemEl2b}). Then $f_j$ is homotopic to $f_j\circ \varphi_1$, $j=1,2,$ and $f_1\circ  \varphi_1$ is homotopic to
 $f_2\circ  \varphi_1$.

For each irreducible holomorphic mapping $f:X\to \mathbb{C}\setminus \{-1,1\}$
and each small enough complex number $\varepsilon$ the holomorphic mapping $(f - i \varepsilon)\mid X^0$ is homotopic to $f\mid X^0$, is irreducible and takes values in $\mathbb{C}\setminus \{-1,1\}$. Hence, the homotopy class of the restriction to $X^0$ of any irreducible holomorphic mapping $X\to \mathbb{C}\setminus \{-1,1\}$ contains an irreducible holomorphic mapping for which zero is a regular value of the imaginary part. This means that the number $N_X$ does not exceed the number $N_{X_0}^*$ of irreducible homotopy classes $X^0\to \mathbb{C}\setminus \{-1,1\}$ that contain a
holomorphic mapping for which $0$ is a regular value of the imaginary part.

We estimate now the number $N_{X_0}^*$. Take any irreducible holomorphic
mapping ${\sf f}:X^0\to \mathbb{C}\setminus\{-1,1\}$ such that  $0$ is a regular value of the imaginary part ${\rm Im}\,{\sf f}$ and apply
Proposition \ref{propfin2}.
We obtain the following objects: a M\"obius transformation $M_l$, that maps one of the components of $\mathbb{R}\setminus \{-1,1\}$ onto $(-1,1)$; further, a point $q\in X^0$ and a smooth curve $\alpha$ in $X^0$ with initial point $q_0$ and terminal point $q$, such that for the mapping
$F=M_l \circ {\sf f}$ the inclusion $q'\stackrel{def}=F(q)\in (-1,1)$ holds, and
for the generators $e_j\stackrel{def}={\rm Is}_{\alpha}(e_{j,0}),\;  e_{j,0}\in \mathcal{E},$ of $\pi_1(X^0,q)$ the inequalities \eqref{eqfin3b} hold with $X$ replaced by $X^0$.
Our goal is to find
a smooth mapping $\hat{F}:X^0\to \mathbb{C}\setminus\{-1,1\}$, that is free homotopic to $F$, maps $q_0$ to $q'$,
and satisfies $(\hat{F})_*(e_0)=F_*(e)$.
Then the inequalities \eqref{eqfin3b}) will imply the inequality
\begin{equation}\label{eqfin3b+}
\mathcal{L}_-((\hat{F})_*(e_0)  ) \leq 12\pi \lambda_7(X^0)
\end{equation}
for each $e_{j,0}\in \mathcal{E}$.

Write $e=\mbox{Is}_{\alpha}(e_{0}) \in \pi_1(X,q)$  for each $e_{0}\in \pi_1(X,q_0)$.
Parameterise $\alpha$ by the interval $[0,1]$.
The  image of $\alpha$
under the mapping $F$ is the curve $\beta= F \circ \alpha$ in $\mathbb{C} \setminus \{-1,1\}$  with initial point $F(q_0)$ and terminal point $F(q)=q'$. Then $F_*(e_{0})=(\mbox{Is}_{\beta})^{-1}(F_*(e))$.
The curve $\beta$ is parameterized by the interval $[0,1]$.
Choose a free homotopy $F_t,\, t \in [0,1],$ of mappings from $X^0$ to $\mathbb{C} \setminus \{-1,1\}$, that changes the mappings only in a small neighbourhood of $\beta$ and joins the mapping $F_0 \stackrel{def}=F$ with a (smooth) mapping $F_1$ denoted by $\hat F$, so that $F_t(q_0)=\beta(t),\, t\in [0,1]$.
The value $\beta(t)$
moves from the point $\beta(0)=F(q_0)$ to $\beta(1)=q'$ along the curve $\beta$.
Denote by $\beta_t$ the curve that runs from $\beta(t)$ to $\beta(1)$ along $\beta$. Then $\beta_0=\beta$ and $\beta_1$ is a constant curve.
Let $\gamma_0$ be a curve that represents $e_0$. Then $\gamma\stackrel{def}={\rm Is}_{\alpha}(\gamma_0)$ represents $e$.
The base point of the curve $F_t(\gamma_0)$ equals $F_t(q_0)=\beta(t)$. Hence, we obtain a continuous family of curves $\beta_t^{-1} F_t(\gamma_0) \beta_t$ with base point $\beta(1)=F(q)$. For $t=1$ the curve is equal to $F_1(\gamma_0)=\hat{F}(\gamma_0)$, for $t=0$ the curve is equal to $\beta^{-1} F_0(\gamma_0) \beta=F_0(\alpha^{-1} \gamma_0\alpha)=F_0( {\rm Is}_{\alpha}(\gamma_0))=           F_0(\gamma)$. Since the two curves
$F_1(\gamma_0)$ and $F_0(\gamma)$
are homotopic and $F_1=\hat F$, $F_0=F$, we obtain $\hat{F}_*(e_0)=F_*(e)$.

We estimate now the number of free homotopy classes of mappings $X^0\to \mathbb{C}\setminus \{-1,1\}$ whose associated conjugacy class of homomorphisms
$\pi_1(X^0,q_0)\to \pi_1( \mathbb{C}\setminus\{-1,1\},q')$
contains a homomorphism
that satisfies the inequalities \eqref{eqfin3b+}. Recall that $\pi_1( \mathbb{C}\setminus\{-1,1\},q')$ is canonically isomorphic to $\pi_1( \mathbb{C}\setminus\{-1,1\},0)$ by the isomorphism $ {\rm Is}_{\beta'}$ where the curve $\beta'$ runs along the segment that joins $q'$ and $0$ and is contained in $(-1,1)$.
By Lemma \ref{lem10.1a}
there are at most
$\frac{1}{2}e^{36\pi \lambda_7(X^0)}+1\leq \frac{3}{2}e^{36\pi \lambda_7(X^0)}$
different reduced words $w \in \pi_1(\mathbb{C}\setminus\{-1,1\}),0)$ (including the identity) with $\mathcal{L}_-(w)\leq 12\pi \lambda_7(X^0)$.
Hence, there are at most $(\frac{3}{2}e^{36\pi  \lambda_7(X^0)})^{2g+m}$ different homomorphisms
$h: \pi_1(X,q_0) \to \pi_1(\mathbb{C}\setminus \{-1,1\},0)$
with $\mathcal{L}_-(h(e_0))\leq 12\pi \lambda_7(X^0)$ for each element $e_0$ of the set of generators $\mathcal{E}$ of $\pi_1(X^0,q_0)$.
By Theorem \ref{thmEl-1} there are at most $(\frac{3}{2}e^{36\pi \lambda_7(X^0)})^{2g+m}$ different
free homotopy classes of mappings $X^0\to \mathbb{C}\setminus \{-1,1\}$,
 whose associated conjugacy class of homomorphisms
$\pi_1(X^0,q_0)\to \pi_1( \mathbb{C}\setminus\{-1,1\},0)$
contains a homomorphism
that satisfies the inequalities \eqref{eqfin3b+}. It follows that $N_{X_0}^*\leq(\frac{3}{2}e^{36\pi \lambda_7(X^0)})^{2g+m}$.

Our arguments show that for each irreducible or contractible holomorphic mapping $f$
on $X$ there are arbitrarily small numbers $\varepsilon$, such that one of the three mappings $M_l(f - i \varepsilon)\mid X^0$,
$l=0,1,2,$ belongs to one of at most $(\frac{3}{2}e^{36\pi \lambda_7(X^0)})^{2g+m}$ free homotopy classes of mappings $ X^0\to \mathbb{C}\setminus \{-1,1\}$. Since for small emough numbers $\varepsilon$ the mappings $M_l(f )\mid X^0$ and  $M_l(f - i \varepsilon)\mid X^0$ are homotopic, the mapping $M_l(f)\mid X^0$ belongs to one of at most $(\frac{3}{2}e^{36\pi \lambda_7(X^0)})^{2g+m}$ free homotopy classes of mappings $ X^0\to \mathbb{C}\setminus \{-1,1\}$.
Compose the mappings of each class with each of the three M\"obius transformations $M_{l'}^{-1}$. We obtain no more than  $3(\frac{3}{2}e^{36\pi \lambda_7(X^0)})^{2g+m}$ free homotopy classes of mappings $ X^0\to \mathbb{C}\setminus \{-1,1\}$, and $f\mid X^0=M_l^{-1}(M_l(f))\mid X^0$ belongs to one of them. Identifying $\pi_1(X^0,q_0)$ with $\pi_1(X,q_0)$, the homotopy classes of mappings $ X\to \mathbb{C}\setminus \{-1,1\}$ have the same associated conjugacy classes of homomorphism $\pi_1(X,q_0) \to \pi_1(\mathbb{C}\setminus \{-1,1\},0)$ as the homotopy classes of their restrictions. We see that there are no more than $3(\frac{3}{2}e^{36\pi \lambda_7(X^0)})^{2g+m}$ holomorphic mappings  $ X\to \mathbb{C}\setminus \{-1,1\}$ that are irreducible or contractible.
Theorem \ref{thmfin1} is proved with the upper bound $3(\frac{3}{2}e^{24 \pi \lambda_7(X^0)})^{2g+m}$ for an arbitrary relatively compact domain $X^0\subset X$ that is a deformation retract of $X$.

It remains to prove that
$$
\lambda_7(X)= \inf\{\lambda_7(X^0): X^0 \Subset X \; \mbox{is\, a\, deformation\, retract\, of\,} X\,\}\,.
$$
We have to prove that for each $e_0 \in \pi_1(X,q_0)$ the quantity
$\lambda(\widetilde{X} \diagup ({\rm Is}^{\tilde{q}_0})^{-1}(\langle e_0 \rangle))$ is
equal to the infimum of $\lambda( \widetilde{X^0}\diagup ({\rm Is}^{\tilde{q}_0})^{-1}(\langle e_0 \rangle))$ over all $X^0$ being open relatively compact subsets of $X$ which are deformation retracts of $X$. Here $\widetilde{X^0}$ is the universal covering of $X^0$, and the fundamental groups of $X$ and $X^0$ are identified. $\widetilde{X^0}$ ($\widetilde{X}$, respectively) can be defined as set of homotopy classes of arcs in $X^0$ (in $X$, respectively) joining $q_0$ with a point $q\in X^0$ (in $X$ respectively) equipped with the complex structure induced by the projection to the endpoint of the arcs, and the point $\tilde{q}_0$ corresponds to the class of the constant curve.
The isomorphism $({\rm Is}^{\tilde{q}_0})^{-1}$ from $\pi_1(X^0,q_0)$ to the group of covering transformations on $\widetilde{X^0}$ is defined in the same way as it was done for $X$ instead of $X^0$.
These considerations imply that there is a holomorphic mapping from $\widetilde{X}^0 \diagup ({\rm Is}^{\tilde{q}_0})^{-1}(\langle e_0 \rangle)$ into $\widetilde{X} \diagup ({\rm Is}^{\tilde{q}_0})^{-1}(\langle e_0 \rangle)$. Hence, the extremal length of the first set is not smaller than the extremal length of the second set,
$\lambda(\widetilde{X}^0 \diagup {\rm Is}^{\tilde{q}_0})^{-1}(\langle e_0 \rangle)\geq \lambda( \widetilde{X} \diagup ({\rm Is}^{\tilde{q}_0})^{-1}(\langle e_0 \rangle)$.

Vice versa, take any annulus $A^0$ that is a relatively compact subset of
$\widetilde{X} \diagup ({\rm Is}^{\tilde{q}_0})^{-1}(\langle e_0 \rangle)$
and is a deformation retract of $\widetilde{X} \diagup ({\rm Is}^{\tilde{q}_0})^{-1}(\langle e_0 \rangle)$.
Its projection to $X$ is relatively compact in $X$, hence, it is contained in a relatively compact open subset  $X^0$ of $X$ that is a deformation retract of $X$. Hence, $A^0$ can be considered as subset of
$\widetilde{X^0} \diagup ({\rm Is}^{\tilde{q}_0})^{-1}(\langle e_0 \rangle)$, and, hence, $\lambda(\widetilde{X^0} \diagup ({\rm Is}^{\tilde{q}_0})^{-1}(\langle e_0 \rangle)) \leq \lambda(A^0)$. Since
\begin{align*}
\lambda (\widetilde{X} \diagup ({\rm Is}^{\tilde{q}_0})^{-1}(\langle e_0 \rangle)) = \inf \{\lambda (A^0): A^0 \Subset \widetilde{X} \diagup ({\rm Is}^{\tilde{q}_0})^{-1}(\langle e_0 \rangle)\nonumber\\
 \mbox{\,is\, a\, deformation\, retract\, of\,}\widetilde{X} \diagup ({\rm Is}^{\tilde{q}_0})^{-1}(\langle e_0 \rangle) \} \,,
\end{align*}
we obtain $\lambda(\widetilde{X}^0 \diagup {\rm Is}^{\tilde{q}_0})^{-1}(\langle e_0 \rangle)\leq \lambda( \widetilde{X} \diagup ({\rm Is}^{\tilde{q}_0})^{-1}(\langle e_0 \rangle)$.
We are done. \hfill $\Box$

\medskip

We proved a slightly stronger statement, namely, the number of homotopy classes of mappings $X\to \mathbb{C}\setminus\{-1,1\}$ that contain a contractible holomorphic mapping or an irreducible holomorphic mapping does not exceed $3(\frac{3}{2}e^{36\pi \lambda_7(X)})^{2g+m}$.

\medskip

\noindent {\bf Proof of Proposition \ref{propfin2a}.}
Denote the zero set $\{x\in X: \mbox{Im}f(x)=0\}$ by $L$. Since $f$ is not contractible,  $L\neq \emptyset$.

\noindent {\bf 1. A torus with a hole.} Assume first that $X$ is a torus with a hole with  base point $q_0$.
For notational convenience we denote by $e_0$ and ${e'_0}$ the two elements of the set of generators $\mathcal{E}$ of $\pi_1(X,q_0)$ that is associated to a standard bouquet of circles for $X$.
We claim that there is a connected component $L_0$ of $L$ such that (after suitable orientation) the intersection number of the free homotopy class  of one of the elements of $\mathcal{E}$,
say of $\widehat{e_0}$, with $L_0$ is positive, and
the intersection number with one of the classes $\reallywidehat{e'_0}$, or $\reallywidehat{(e'_0)^{-1}}$, or $\reallywidehat{e_0 \,e'_0}$ with $L_0$ is positive.

The claim is easy to prove in the case when there is a component $L_0$ of $L$ which is a simple closed curve that is not contractible and not contractible to the hole of $X$. Indeed, consider the inclusion of $X$ into a closed torus $X^c$ and the homomorphism on fundamental groups $\pi_1(X,q_0)\to \pi_1(X^c,q_0)$ induced by the inclusion. Denote by $e_0^c$ and ${e'_0}^c$ the images of $e_0$ and $e_0'$ under this homomorphism. Notice that $e_0^c$ and ${e'_0}^c$ commute. The (image under the inclusion of the) curve $L_0$ is a simple closed non-contractible curve in $X^c$. It represents the free homotopy class of an element $(e_0^c)^j({e'_0}^c)^k$ for some integers $j$ and $k$ which are not both equal to zero. Hence, $L_0$ is not null-homologous in $X^c$, and by the Poincar\'{e} Duality Theorem
for one of the generators, say for $e_0^c$, the representatives of the free homotopy class $\reallywidehat{e_0^c}$  have non-zero intersection number 
with $L_0$.
After possibly reorienting $L_0$, we may assume that this intersection number is positive. There is a representative of the class $\reallywidehat{e_0^c}$ which is contained in $X$, hence,
$\widehat{e_0}$ has positive intersection number with $L_0$.

Suppose all compact connected components of $L$ are contractible or
contractible to the hole of $X$.
Consider a relatively compact domain ${X}''\Subset {X}$ in ${X}$ with smooth boundary which is a deformation retract of ${X}$ such that for each connected component of $L$ at most one component of its intersection with ${X}''$ is not contractible to the hole of ${X}''$. (See the paragraph on ''Regular zero sets''.) There is at least one component of $L\cap {X}''$ that is not contractible to the hole of ${X}''$. Indeed, otherwise the free homotopy class of each element of $\mathcal{E}$ could be represented by a loop avoiding $L$, and, hence, the monodromy of $f$ along each
element of $\mathcal{E}$ would be conjugate to the identity, and, hence, equal to the identity, i.e. contrary to the assumption, $f:X\to \mathbb{C}\setminus \{-1,1\} $ would be free homotopic to a constant.

Take a component $L_0''$ of $L\cap {X}''$ that is not contractible to the hole of ${X}''$.
There is an arc of $\partial {X}''$ between the endpoints of $L_0''$ such that
the union $\tilde{L}_0$ of the component $L_0''$ with this arc is a closed curve in ${X}$ that is
not contractible and not contractible to the hole. Hence, for one of the elements of $\mathcal{E}$, say for $e_0$, the intersection number of the free homotopy class $\widehat{e_0}$ with the closed curve $\tilde{L}_0$ is positive after orienting the curve $\tilde{L}_0$ suitably. We may take a representative $\gamma$ of $\widehat{e_0}$ that is
contained in ${X}''$. Then $\gamma$ has positive intersection number with $L_0''$.
Denote the connected component of $L$ that contains $L_0''$ by $L_0$. All components of $L_0\cap {X}''$ different from $L_0''$ are contractible to the hole of ${X}''$. Hence, $\gamma$ has intersection number zero with each of these components.
Hence, $\gamma$ has positive intersection number with $L_0$ since ${\gamma}\subset {X}''$.
We proved that the class  $\widehat{e_0}$
has positive intersection number with $L_0$.

If $\reallywidehat{e'_0}$ also has non-zero intersection number with $L_0$  we define $e''_0=(e'_0)^{\pm 1}$ so that the intersection number of  $\reallywidehat{e''_0}$ with $L_0$ is positive. If $\reallywidehat{e'_0}$  has zero intersection number with $L_0$ we put $e''_0=e_0 \,e'_0$. Then again the intersection number of $\reallywidehat{e''_0}$ with $L_0$ is positive.  Also, the intersection number of  $\reallywidehat{e_0 \, e''_0}$ with $L_0$ is positive. The set $\mathcal{E}_2'\stackrel{def}=\{e_0,e_0''\}$ satisfies the condition required in the proposition.
We obtained Proposition \ref{propfin2a} for a torus with a hole.

\noindent {\bf 2. A planar domain.} Let $X$ be a planar domain.
The domain $X$ is conformally equivalent to  a disc with $m$ smoothly bounded holes,
equivalently, to the Riemann sphere with $m+1$ smoothly bounded holes, $\mathbb{P}^1 \setminus \bigcup_{j=1}^{m+1}\mathcal{C}_j$, where $\mathcal{C}_{m+1}$ contains the point $\infty$. As before the base point of $X$ is denoted by $q_0$, and for each $j=1,\ldots,m,$ the generator $e_{j,0}\in \mathcal{E}\subset \pi_1(X,q_0)$
is represented by a curve that surrounds $\mathcal{C}_j$ once counterclockwise.
Since $f$ is not contractible, there must be a connected component of $L$ that has limit points on some $\mathcal{C}_j$ with $j\leq m$. Indeed, otherwise
the free homotopy class of each  generator could be represented by a curve that avoids $L$.
This would imply that all monodromies are equal to the identity.
We claim that
there exists a component $L_0$ of $L$
with limit points on the boundary of two components $\partial {\mathcal{C}}_{j'}$ and
$\partial {\mathcal{C}}_{j''}$ for some $j', j'' \in \{1,\ldots,m+1\}$ with $j''\neq j'$.

Indeed, assume the contrary.
Then, if a component of $L$ has limit points on a component $\partial {\mathcal{C}}_{j},\, j\leq m,$ then all its limit points are on $\partial {\mathcal{C}}_{j}$.
Take a smoothly bounded simply connected domain ${\mathcal{C}}'_{j}\Subset X\cup {{\mathcal{C}}_{j}}$
that contains the closed set ${{\mathcal{C}}_{j}} $ , so that its boundary $\partial {\mathcal{C}}'_{j}$ represents $\reallywidehat{e_{j,0}}$.
Then all components $L'_k$ of $L\setminus {\mathcal{C}}'_{j} $ with an endpoint on $\partial \mathcal{C}'_{j}$ have another endpoint on this circle. The two endpoints of $L_k'$ on $\partial {\mathcal{C}}'_{j}$ divide $\partial {\mathcal{C}}'_{j}$ into two connected components.
The union of $\overline{L_k'}$ with each of the two components of $\partial {\mathcal{C}}'_{j}\setminus \overline{L_k'}$ is a simple closed curve in $\mathbb{C}$, and, hence, by the Jordan Curve Theorem it bounds a relatively compact topological disc in $\mathbb{C}$. One of these discs contains $ {\mathcal{C}}'_{j}$, the other does not.
Assign to each component $L_k'$ of $L \setminus {\mathcal{C}}'_{j}   $ with both endpoints on $\partial \mathcal{C}'_{j} $ the closed arc $\alpha_k$  in $\partial \mathcal{C}'_{j}$ with the same endpoints as $L_k'$, whose union with $L_k'$ bounds a relatively compact topological disc in $\mathbb{C}$ that does not contain $\mathcal{C}'_{j}$. These discs are partially ordered by inclusion,
since the $L_k'$ are pairwise disjoint. Hence, the arcs $\alpha_k$ are partially ordered by inclusion. For an arc $\alpha_k$  which contains no other of the arcs (a minimal arc) the curve $f \circ \alpha_k$ except its endpoints is contained in  $\mathbb{C}\setminus \mathbb{R}$. Moreover,
the endpoints of  $f \circ \alpha_k$ lie on ${f(\overline{L_k'})}$, which is contained in one connected components of $\mathbb{R}\setminus \{-1,1\}$, since $\overline{L_k'}$ is connected. Hence, the curve $f \circ \alpha_k$ is homotopic  in $\mathbb{C}\setminus \{-1,1\}$ (with fixed endpoints) to a curve in $\mathbb{R} \setminus \{-1,1\}$. The function $f$ either maps all points on $\partial \mathcal{C}'_{j}\setminus \alpha_k$ that are close to $\alpha_k$ to the open upper half-plane or maps them all to the open lower half-plane. (Recall, that zero is a regular value of $\mbox{Im}f$.) Hence, for an open arc $\alpha'_k \subset \partial \mathcal{C}'_{j} $ that contains $\alpha_k$ the curve $f \circ \alpha'_k $ is homotopic  in  $\mathbb{C}\setminus \{-1,1\}$ (with fixed endpoints) to a curve in $\mathbb{C}\setminus \mathbb{R}$.

Consider the arcs $\alpha_k$ with the following property. For an open arc $\alpha'_k$  in $\partial \mathcal{C}'_{j} $ which contains the closed arc $\alpha_k$ the mapping $f \circ \alpha'_k$ is homotopic in $\mathbb{C}\setminus \{-1,1\}$ (with fixed endpoints) to a curve contained in  $\mathbb{C}\setminus \mathbb{R}$. Induction on the arcs by inclusion shows that this property is satisfied for all maximal arcs among the $\alpha_k$ and, hence, $f \mid \partial \mathcal{C}'_{j} $ is contractible in  $\mathbb{C}\setminus \{-1,1\}$.
Hence, if the claim was not true, then for each hole $\mathcal{C}_j,\, j\leq m$, whose boundary contains limit points of a connected component of $L$, the monodromy along the curve $\mathcal{C}'_j$ (with any base point contained in $\mathcal{C}'_j$) that represents $\reallywidehat{e_{j,0}}$
would be trivial. Then all monodromies would be trivial, which contradicts the fact that the mapping is not contractible. The contradiction proves the claim.

With $j'$ and $j''$ being the numbers of the claim and $j'\leq m$ we consider the set $\mathcal{E}_2' \subset \mathcal{E}_2$  which consists of the following primitive elements:
$e_{j',0}$, the element $(e_{j'',0})^{-1}$ provided $j'' \neq m+1$, and  $e_{j',0}\, e_{j,0}$ for all $j=1,\ldots,m,\, j\neq j', j \neq j''$. The free homotopy class of each element of $\mathcal{E}_2'$  has intersection number $1$ with $L_0$ after suitable orientation of the curve $L_0$.
Each product of at most two different elements of $\mathcal{E}_2'$ is a primitive element of $\pi_1(X,q)$ and is contained in $\mathcal{E}_4$. Moreover, the intersection number with $L_0$ of the free homotopy class of each product of two different elements of $\mathcal{E}_2'$
equals $2$.
Each element of $\mathcal{E}$ is the product of at most two elements of $\mathcal{E}_2'\cup (\mathcal{E}_2')^{-1}$.

The proposition is proved for the case of planar domains $X$.
\hfill $\Box$

\medskip

\noindent {\bf Proof of Proposition \ref{propfin2}.}

\noindent {\bf 1.  A torus with a hole.} Suppose $X$ is a torus with a hole.
Consider the curve $L_0$ and the set $\mathcal{E}_2' \subset \pi_1(X,q_0)$ obtained in Proposition \ref{propfin2a}. For one of the functions $M_l \circ f$, denoted by $F$, the image $F (L_0)$ is contained in $(-1,1)$.
Let $e_0$, $e_0'$ be the two elements of $\mathcal{E}$.
Move the base point $q_0$ to a point $q\in L_0$ along a curve $\alpha$ in $X$, and consider the generators $e=\mbox{Is}_{\alpha} (e_0)$ and
$e'=\mbox{Is}_{\alpha} ({e'_0})$ of $ \pi_1(X,q)$, and the set $\mbox{Is}_{\alpha}(\mathcal{E}_2')\subset \pi_1(X,q)$. Then $e$ and $e'$ are products of at most two elements of $\mbox{Is}_{\alpha}(\mathcal{E}_2')$.
Since the free homotopy class of an element of $\pi_1(X,q_0)$ coincides with the free homotopy class of the element of $\pi_1(X,q)$ obtained by applying $\mbox{Is}_{\alpha}$,
the free homotopy class of each product of one or two elements of $\mbox{Is}_{\alpha}(\mathcal{E}_2')$ intersects $L_0$.
We may assume as in the proof of Proposition \ref{propfin2a} that  $\mbox{Is}_{\alpha}(\mathcal{E}_2')$ consists of the elements $e$ and $e''$, where $e''$ is either equal to ${e'}^{\pm 1}$, or equals the product of $e$ and $e'$.
Lemma \ref{lemfin3} applies to the pair $e$, $e''$,
the function $F$,
and the curve $L_0$.
Since $F$ is irreducible, the monodromies of $F$ along
$e$ and $e''$ are not powers of a single standard generator of the fundamental group of $\pi_1(\mathbb{C}\setminus \{-1,1\},q')$.
Hence, the monodromy along each of the $e$ and $e''$ is the product of at most two elements of $\mathcal{L}_-$ not exceeding $2 \pi\lambda_{e,e''}$. Therefore,
the monodromy of $F$ along each of the $e$ and $e''$ has $\mathcal{L}_-$ not exceeding $4 \pi\lambda_{e,e''}$. Notice that $\lambda_{e,e''}=\lambda_{e_0 e_0''}\leq \lambda_3(X)$, since $e_0''$ is the product of at most two factors, each an element of $\mathcal{E}\cup \mathcal{E}^{-1}$. Since $e'$ is the product of at most two different elements among the $e$ and $e''$ and their inverses, we obtain Proposition \ref{propfin2} for $e$ and $e'$, in particular  $\mathcal{L}_-(F_*(e))$ and $\mathcal{L}_-(F_*( e'))$ do not exceed $8\pi \lambda_3(X)$.
Proposition \ref{propfin2} is proved for tori with a hole.

\noindent {\bf 2. A planar domain.} Consider the curve $L_0$ and the set $\mathcal{E}_2'$ of Proposition \ref{propfin2a}.
Move the base point $q_0$ along an arc $\alpha$ to a point $q \in L_0$. Then $f(q) \in \mathbb{R}\setminus \{-1,1\}$ and for one of the mappings $f$, $M_1 \circ f$, or $M_2 \circ f$, denoted by $F$, the inclusion $F(L_0)\subset (-1,1))$ holds, hence, $q'\stackrel{def}=F(q)$ is contained in $(-1,1)$. Put  $e_j =  \mbox{Is}_{\alpha}(e_{j,0})$ for each element $e_{j,0}\in \mathcal{E}$. The $e_j$ form the basis $\mbox{Is}_{\alpha}(\mathcal{E})$ of $\pi_1(X,q)$. The set $\mbox{Is}_{\alpha}(\mathcal{E}_2')$ consists of primitive elements of $\pi_1(X,q)$ such that the free homotopy class of each product of one or two elements of $\mbox{Is}_{\alpha}(\mathcal{E}_2')$ intersects $L_0$. Moreover, each element of $\mbox{Is}_{\alpha}(\mathcal{E})$ is the product of one or two elements of  $\mbox{Is}_{\alpha}(\mathcal{E}_2') \cup (\mbox{Is}_{\alpha}(\mathcal{E}_2'))^{-1}$.

By the condition of the proposition not all monodromies $F_*(e),\, e \in \mbox{Is}_{\alpha}(\mathcal{E}_2'),$ are (trivial or non-trivial) powers of the same standard generator of $\pi_1(\mathbb{C}\setminus \{-1,1\},q')$.
Apply Lemma \ref{lemfin3} to all pairs of elements of $\mbox{Is}_{\alpha}(\mathcal{E}_2')$  whose monodromies are not (trivial or non-trivial) powers of the same standard generator of $\pi_1(\mathbb{C} \setminus \{-1,1\},q')$.
Since the product of at most two different elements of $\mbox{Is}_{\alpha}(\mathcal{E}_2')$ is contained in  $\mbox{Is}_{\alpha}(\mathcal{E}_4)$, Lemma \ref{lemfin3} shows that the monodromy $F_*(e)$ along each element  $e \in \mbox{Is}_{\alpha}(\mathcal{E}_2')$ is the product of at most two factors, each with $\mathcal{L}_-$ not exceeding $2 \pi \lambda_4(X)$.
Since each element of $\mbox{Is}_{\alpha}(\mathcal{E}) $ is a product of at most two factors in $\mathcal{E}_2' \cup (\mathcal{E}_2')^{-1}$,
the  monodromy $F_*(e_j)$ along each generator $e_j$ of $\pi_1(X,q)$ is the product of at most $4$ factors of $\mathcal{L}_-$ not exceeding  $2 \pi \lambda_4(X)$, and, hence, each monodromy $F_*(e_j)$ has $\mathcal{L}_-$ not exceeding $8 \pi \lambda_4(X)$.
Proposition \ref{propfin2} is proved for planar domains.

\noindent {\bf 3.1. The general case. Diagrams of coverings.}
We will use diagrams of coverings to reduce this case to the case of a torus with a hole or to the case of the Riemann sphere with three holes.

Let as before $\tilde{q}_0$ be the point in $\tilde X$ with ${\sf P}(\tilde{q}_0)=q_0$
chosen in Section \ref{sec:2.0}. Let $N$ be a subgroup of the fundamental group $\pi_1(X,q_0)$ and let $\omega^N:\tilde{X}\to \tilde{X}\diagup({\rm Is}^{\tilde{q}_0})^{-1}(N)=X(N)$ be the projection defined in Section \ref{sec:2.0}. Put $(q_0)_N \stackrel{def}= \omega^N(\tilde{q}_0)$. For an element $e_0\in N\subset \pi_1(X,q_0)$ we denote by $(e_0)_N$ the element of $\pi_1(X(N),(q_0)_N)$ that is obtained as follows. Take a curve $\gamma$ in $X$ with base point $q_0$ that represents $e_0\in N$. Let
$\tilde \gamma$ be its lift to $\tilde X$ with initial point $\tilde{q}_0$.
Then $\gamma_N\stackrel{def}=\omega^N(\tilde{\gamma})$ is a closed curve in $X(N)=\tilde{X}\diagup ({\rm Is}^{\tilde{q}_0})^{-1}(N) $ with base point $(q_0)_N$. The element of $\pi_1(X(N), (q_0)_N)$ represented by $\gamma_N$ is the required element $(e_0)_N$. All curves $\gamma'_N$  representing $(e_0)_N$ have the form $\omega^N(\tilde{\gamma}')$ for a curve $\tilde{\gamma}'$ in $\tilde X$ with  initial point $\tilde{q}_0$ and terminal point $({\rm Is}^{\tilde{q}_0})^{-1}(e_0)(\tilde{q}_0)$.
Since $\omega_N\circ\omega^N={\sf P}$, the curve $\omega_N(\gamma'_N)={\sf P}(\tilde{\gamma}')={\gamma}'$
represents $e_0$ for each curve $\gamma'_N$ in $X(N)$ that represents $(e_0)_N$. We obtain $(\omega_N)_*((e_0)_N)=e_0$. For two subgroups $N_1\leq N_2$ of $\pi_1(X,q_0)$ we obtain $(\omega^{N_2}_{N_1})_*((e_0)_{N_1})=(e_0)_{N_2}$, $e_0\in N_1$ (see the commutative diagram Figure \ref{fig3}). 

Let $\tilde q$ be another base point of $\tilde X$ and let $\tilde \alpha$
be a curve in $\tilde X$ with initial point $\tilde{q}_0$ and terminal point $\tilde q$. Let again $N$ be a subgroup of $\pi_1(X,q_0)$.  Put $q_N\stackrel{def}=\omega^N(\tilde{q})$.
The curve $\alpha_N=\omega^N(\tilde{\alpha})$ in $X(N)$, and the base point $\tilde q$ of $\tilde X$ are compatible, hence,
$X({\rm Is}_{\alpha_N}(N))\stackrel{def}=\tilde{X}\diagup ({\rm Is}^{\tilde{q}})^{-1}({\rm Is}_{\alpha_N}(N))=\tilde{X}\diagup ({\rm Is}^{\tilde{q}_0})^{-1}( N  )=X(N)$.

We will use the previous notation $\omega^{N_2}_{N_1}$ also for the projection
$$
\tilde{X}\diagup ({\rm Is}^{\tilde{q}})^{-1}({\rm Is}_{\alpha_{N_1}}(N_1))\to \tilde{X}\diagup ({\rm Is}^{\tilde{q}})^{-1}({\rm Is}_{\alpha_{N_2}}(N_2))\,,
$$
with $N_1\leq N_2$ being subgroups of $\pi_1(X,q_0)$ ($N_1$ may be the identity and $N_2$ may be $\pi_1(X,q_0)$.)

Put $\alpha\stackrel{def}={\sf P}(\tilde{\alpha})$. For an element $e_0\in \pi_1(X,q_0)$ we put $e\stackrel{def}={\rm Is}_{\alpha}(e_0)\in \pi_1(X,q)$ and denote by $e_N$ the element of $\pi_1(X(N),q_N)$, that is represented by $\omega^N(\tilde{\gamma})$ for a curve $\tilde \gamma$ in $\tilde X$ with initial point $\tilde q$ and projection ${\sf P}(\tilde{\gamma})$ representing $e$. Again $(\omega^{N_2}_{N_1})_*(e_{N_1})=e_{N_2}$ for subgroups $N_1\leq N_2$ of $\pi_1(X,q)$ and $e\in N_1$, in particular $( \omega_{N})_*(e_{N})=e$ for a subgroup $N$ of $\pi_1(X,q)$ and $e\in N$.

\noindent {\bf 3.2. The estimate for a chosen pair of monodromies.}
Since the mapping $f:X\to \mathbb{C}\setminus\{-1,1\}$ is irreducible, there exist two elements $e_0',\, e_0''\in \mathcal{E}\subset\pi_1(X,q_0)$ such that the monodromies $f_*(e_0')$ and $f_*(e_0'')$ are not powers of a single conjugate of a power of one of the elements $a_1$, $a_2$ or $a_1a_2$. The fundamental group of the Riemann surface $X(\langle e_0',e_0''\rangle)$ is a free group in the two generators $(e'_0)_{\langle e_0',e_0''\rangle}$ and $(e''_0)_{\langle e_0',e_0''\rangle}$, hence, $X(\langle e_0',e_0''\rangle)$ is either a torus with a hole or is equal to $\mathbb{P}^1$ with three holes. Moreover, the system $\mathcal{E}_{\langle e_0',e_0''\rangle}=\{(e'_0)_{\langle e_0',e_0''\rangle}, \,(e''_0)_{\langle e_0',e_0''\rangle}\}$  of generators of the fundamental group $\pi_1(X(\langle e_0',e_0''\rangle),(q_0)_{\langle e_0',e_0''\rangle})$ is associated to a standard bouquet of circles for $X(\langle e_0',e_0''\rangle)$.
This can be seen as follows. The set of generators $\mathcal{E}$ of $\pi_1(X,q_0)$ is associated to a standard bouquet of circles for $X$. For each $e_0\in \mathcal{E}$ we denote the circle of the bouquet that represents $e_0$ by $\gamma_{e_0}$.
For each $e_0\in\mathcal{E}$ we lift the circle $\gamma_{e_0}$ with base point $q_0$ to an arc $\widetilde{\gamma_{e_0}}$ in
$\tilde X$ with initial point
$\tilde{q_0}$. Let $D$ be a small disc in $X$ around $q_0$, and $\widetilde{D_0}$, $\widetilde{D_{e_0}}, e_0\in \mathcal{E},$ be the preimages of $D$ under the projection
${\sf P}:\tilde{X}\to X$,
that contain $\tilde{q_0}$, or the terminal point of $\widetilde{\gamma_{e_0}}$, respectively.
We assume that $D$ is small enough so that the mentioned preimages of $D$
are pairwise disjoint. Put $D_{\langle e'_0,e''_0\rangle}=\omega^{\langle e'_0,e''_0\rangle}(\widetilde{D_0})$.
For $e_0\neq e'_0,e''_0$
the image
$\omega^{\langle e'_0,e''_0\rangle}(\widetilde{D_0}\cup \widetilde{\gamma_{e_0}}\cup \widetilde{D_{e_0}} )$
is the union of an arc $\omega^{\langle e'_0,e''_0\rangle} (\widetilde{\gamma_{e_0}})$ in $X(\langle e'_0,e''_0\rangle)$ with two disjoint discs, each containing an endpoint of the arc and one of them equal to $D_{\langle e'_0,e''_0\rangle}$.
For $e_0=e'_0,e''_0$ the image $\omega^{\langle e',e''\rangle}(\widetilde{D_0}\cup \widetilde{\gamma_{e_0}}\cup \widetilde{D_{e_0}} )$ is the union of $D_{\langle e'_0,e''_0\rangle}$ with the loop $(\gamma_{e_0})_{\langle e'_0,e''_0\rangle}\stackrel{def}= \omega^{\langle e'_0,e''_0\rangle}( \widetilde{\gamma_{e_0}})$. For $e_0=e'_0,e''_0$ the loop $(\gamma_{e_0})_{\langle e'_0,e''_0\rangle}$ in $X(\langle e'_0,e''_0\rangle)$
has base point $(q_0)_{\langle e'_0,e''_0\rangle}=\omega^{\langle e'_0,e''_0\rangle}(\tilde{q_0})$
and represents the generator $(e_0)_{\langle e'_0,e''_0\rangle}$ of the fundamental group of $\pi_1(X(\langle e'_0,e''_0\rangle),(q_0)_{\langle e'_0,e''_0\rangle})$.

Since
the bouquet of circles $\cup_{e_0\in \mathcal{E}}\, \gamma_{e_0}$ is a standard bouquet of circles for $X$, the union $(\gamma_{e'_0})_{\langle e'_0,e''_0\rangle} \cup (\gamma_{e''_0})_{\langle e'_0,e''_0\rangle}$ is a standard bouquet of circles in $X(\langle e'_0,e''_0\rangle)$.  This can be seen by looking at the intersections of the loops with a circle that is contained in $D_{\langle e'_0,e''_0\rangle}$ and surrounds $(q_0)_{\langle e'_0,e''_0\rangle}$. By the commutative diagram of coverings the intersection behaviour is the same as for the images of these objects under $\omega_{\langle e'_0,e''_0\rangle}$.
Hence, since $(\gamma_{e'_0})_{\langle e'_0,e''_0\rangle}$ and $(\gamma_{e''_0})_{\langle e'_0,e''_0\rangle}$ represent the generators $({e'_0})_{\langle e'_0,e''_0\rangle}$ and  $({e''_0})_{\langle e'_0,e''_0\rangle}$ of $\mathcal{E}_{\langle e_0',e_0''\rangle}$, the
union $(\gamma_{e'_0})_{\langle e'_0,e''_0\rangle} \cup (\gamma_{e''_0})_{\langle e'_0,e''_0\rangle}$ is a standard bouquet of circles for $X(\langle e'_0,e''_0\rangle)$.

The set $X(\langle e_0',e_0''\rangle)$ is either a torus with a hole or is equal to $\mathbb{P}^1$ with three holes.
Apply Proposition \ref{propfin2a} to the Riemann surface $X(\langle e_0',e_0''\rangle)$ with base point $(q_0)_{\langle e'_0,e''_0\rangle}$, the holomorphic mapping $f_{\langle e_0',e_0''\rangle}=f\circ\omega_{\langle e_0',e_0''\rangle}$ into $\mathbb{C}\setminus \{-1,1\}$, and the set of generators $\mathcal{E}_{\langle e_0',e_0''\rangle}$ of the fundamental group $\pi_1(X(\langle e'_0,e''_0\rangle),(q_0)_{\langle e'_0,e''_0\rangle})$. We obtain a relatively closed curve $L_{\langle e_0',e_0''\rangle}$ on which the function $f_{\langle e_0',e_0''\rangle}$ is real,
and a set $(\mathcal{E}_{\langle e_0',e_0''\rangle})_2'=\{({\sf e}_0')_{\langle e_0',e_0''\rangle},({\sf e}_0'')_{\langle e_0',e_0''\rangle}\}$ which contains one of the elements of $\mathcal{E}_{\langle e_0',e_0''\rangle}$. The second element of  $(\mathcal{E}_{\langle e_0',e_0''\rangle})_2'$ is either equal to the second element of $\mathcal{E}_{\langle e_0',e_0''\rangle}$ or to its inverse, or to the product of the two elements (in any order) of $\mathcal{E}_{\langle e_0',e_0''\rangle}$.  (We will usually refer to the product $ ({ e}_0')_{\langle e_0',e_0''\rangle}\,({e}_0'') _{\langle e_0',e_0''\rangle}=({ e}_0'\,{ e}_0'')_{\langle e_0',e_0''\rangle} $,
but we may change the product $({ e}_0'\,{ e}_0'')_{\langle e_0',e_0''\rangle} $ to the product $({ e}_0''\,{ e}_0')_{\langle e_0',e_0''\rangle} $, without changing
the estimate of the $\mathcal{L}_-$ of the monodromies of the elements of $\mathcal{E}_2'$.)
The free homotopy classes $\reallywidehat{ ({\sf e}_0')_{\langle e_0',e_0''\rangle}}$, $\reallywidehat{ ({\sf e}_0'')_{\langle e_0',e_0''\rangle}}$, and $\reallywidehat{ ({\sf e}_0')_{\langle e_0',e_0''\rangle}\,({\sf e}_0'') _{\langle e_0',e_0''\rangle}}=\reallywidehat{({\sf e}_0'\,{\sf e}_0'')_{\langle e_0',e_0''\rangle}} $ intersect $L_{\langle e_0',e_0''\rangle}$.

Choose a point $q_{\langle e_0',e_0''\rangle}\in L_{\langle e_0',e_0''\rangle}$ and a point $\tilde{q}\in\tilde{X}$ with $\omega^{\langle e_0',e_0''\rangle}(\tilde{q})= q_{\langle e_0',e_0''\rangle}$. Let $\tilde{\alpha}$ be a curve in $\tilde X$ with initial point $\tilde{q}_0$ and terminal point $\tilde q$, and $\alpha_{\langle e_0',e_0''\rangle}= \omega^{\langle e_0',e_0''\rangle}(\tilde{\alpha})$. Put ${\sf e}'_{\langle e_0',e_0''\rangle}={\rm Is}_{\alpha_{\langle e_0',e_0''\rangle}}(({\sf e}_0')_{\langle e_0',e_0''\rangle})$ and ${\sf e}''_{\langle e_0',e_0''\rangle}={\rm Is}_{\alpha_{\langle e_0',e_0''\rangle}}(({\sf e}_0'')_{\langle e_0',e_0''\rangle})$.
For one out of three M\"obius transformations $M_l$ the mapping $F_{\langle e_0',e_0''\rangle}=M_l\circ f_{\langle e_0',e_0''\rangle}=M_l\circ f\circ\omega_{\langle e_0',e_0''\rangle}$ takes $L_{\langle e_0',e_0''\rangle}$ to $(-1,1)$, and hence $F_{\langle e_0',e_0''\rangle}$
takes a value $q'=F_{\langle e_0',e_0''\rangle}(q_{\langle e_0',e_0''\rangle})\in (-1,1)$ at $q_{\langle e_0',e_0''\rangle}$. By Lemma \ref{lemfin3} each of the $(F_{\langle e_0',e_0''\rangle})_*({\sf e}'_{\langle e_0',e_0''\rangle})$ and $(F_{\langle e_0',e_0''\rangle})_*({\sf e}''_{\langle e_0',e_0''\rangle})$ is the product of at most two elements of  $\pi_1(\mathbb{C}\setminus\{-1,1\},q')$     of $\mathcal{L}_-$ not exceeding $2\pi \lambda_3(X(\langle e_0',e_0''\rangle))$, hence,
\begin{align*}
\mathcal{L}_-&((F_{\langle e_0',e_0''\rangle})_*({\sf e}'_{\langle e_0',e_0''\rangle}))\leq 4\pi\lambda_3(X(\langle e_0',e_0''\rangle)),\\
\mathcal{L}_-&((F_{\langle e_0',e_0''\rangle})_*({\sf e}''_{\langle e_0',e_0''\rangle}))\leq 4\pi\lambda_3(X(\langle e_0',e_0''\rangle))\,.
\end{align*}
It follows that each of the $(F_{\langle e_0',e_0''\rangle})_*({ e}'_{\langle e_0',e_0''\rangle})$ and $(F_{\langle e_0',e_0''\rangle})_*({ e}''_{\langle e_0',e_0''\rangle})$ is the product of at most four elements of  $\pi_1(\mathbb{C}\setminus\{-1,1\},q')$     of $\mathcal{L}_-$ not exceeding $2\pi \lambda_3(X(\langle e_0',e_0''\rangle) )$, hence,
\begin{align*}
\mathcal{L}_-&((F_{\langle e_0',e_0''\rangle})_*({ e}'_{\langle e_0',e_0''\rangle}))\leq 8\pi\lambda_3(X(\langle e_0',e_0''\rangle)),\\
\mathcal{L}_-&((F_{\langle e_0',e_0''\rangle})_*({ e}''_{\langle e_0',e_0''\rangle}))\leq 8\pi \lambda_3(X(\langle e_0',e_0''\rangle))\,.
\end{align*}
It remains to take into account that
for a subgroup $N$ of $\pi_1(X,q_0)$
the equation $(F_N)_*({{ e}}_N)=F_*({e})$ holds for each ${e}\in {\rm Is}_{\alpha}(N)$, and
$\lambda_j(X(N))\leq \lambda_j(X)$ for each natural number $j$.

\noindent {\bf 3.3. Other generators. Intersection of free homotopy classes with a component of the zero set.}
Take any element $e\in {\rm Is}_{\alpha}(\mathcal{E})$  that is not in $\langle e',e''\rangle$.
Then the monodromy $F_*(e)$ is either equal to the identity, or
one of the pairs  $(F_*(e),F_*({\sf e}'))  $ or  $(F_*(e),F_*({\sf e}''))  $ consists of two elements of $\pi_1(\mathbb{C}\setminus\{-1,1\},q')$ that are not powers of the same standard generator $a_j$, $j=1$ or $2$. We assume that the second option holds.
Interchanging if necessary ${\sf e}'$ and ${\sf e}''$, we may suppose this option holds for the pair
$(F_*(e),F_*({\sf e}'))  $. Moreover, changing if necessary ${\sf e}'$ to its inverse $({\sf e}')^{-1}$, we may assume that ${\sf e}'$ is either an element of ${\rm Is}_{\alpha}(\mathcal{E})$ or it is a product of two elements of ${\rm Is}_{\alpha}(\mathcal{E})$. The quotient $X(\langle e,{\sf e}'\rangle)=\tilde{X}\diagup ({\rm Is}^{\tilde q})^{-1}(\langle e, {\sf e}'\rangle)$ is a Riemann surface whose fundamental group is a free group in two generators. Hence  $X(\langle e,{\sf  e}'\rangle)$ is either a torus with a hole or is equal to $\mathbb{P}^1$ with three holes.

We consider a diagram of coverings as follows. Let first $X(\langle {\sf  e}'\rangle)= \tilde{X}\diagup ({\rm Is}^{\tilde q})^{-1}(\langle {\sf e}'\rangle)$ be the annulus
with base point $q_{\langle {\sf e}'\rangle}=\omega^{\langle {\sf e}'\rangle}(\tilde{q})$,    that admits a mapping $\omega_{\langle {\sf e}'\rangle}: X(\langle {\sf  e}'\rangle)\to X$ that represents ${\sf  e}'$. By Lemma \ref{lemfin1} with $X$ replaced by $X(\langle {\sf e}', {\sf e}''\rangle)$
the connected component $L_ {\langle {\sf e}'\rangle}$ of $(\omega_{\langle {\sf e}'\rangle}^{\langle {\sf e}', {\sf e}''\rangle}    )^{-1}(L_{\langle {\sf e}', {\sf e}''\rangle})$ that contains $q_{\langle {\sf e}'\rangle}=\omega^{\langle {\sf e}'\rangle}(\tilde{q})$ is a relatively closed curve in $X(\langle {\sf  e}'\rangle)$ with limit points on both boundary components. The free homotopy class of the generator ${\sf  e}'_{\langle {\sf e}'\rangle}$ of $\pi_1(X(\langle {\sf e}'\rangle),q_{\langle {\sf e}'\rangle})$   intersects $L_ {\langle {\sf e}'\rangle}$. The mapping $F_{\langle {\sf e}'\rangle}=M_l \circ f\circ \omega_{\langle {\sf e}'\rangle} $ maps  $L_ {\langle {\sf e}'\rangle}$ into $(-1,1)$, and $F_{\langle {\sf e}'\rangle}(q_{\langle {\sf e}'\rangle})= F\circ{\sf P}(\tilde{q})=q'$.

Next we consider the quotient $X(\langle {\sf  e}', e \rangle)= \tilde{X}\diagup ({\rm Is}^{\tilde q})^{-1}(\langle {\sf e}',e\rangle)$ whose fundamental group is again a free group in two generators. The image $L_{\langle {\sf e}', e\rangle}\stackrel{def}= \omega_{\langle {\sf e}'\rangle}^{\langle {\sf e}', e\rangle}(L_{\langle {\sf e}'\rangle})$ is a connected component of the preimage of $(-1,1)$ under $ F_{\langle {\sf e}', e\rangle}=F\circ \omega_{\langle {\sf e}', e\rangle}$. Indeed, $L_{\langle {\sf e}', e\rangle}$ is connected as image of a connected set under a continuous mapping, and $F_{\langle {\sf e}', e\rangle}(L_{\langle {\sf e}', e\rangle})=
F_{\langle {\sf e}', e\rangle}(\omega_{\langle {\sf e}'\rangle}^{\langle {\sf e}', e\rangle}(L_{\langle {\sf e}'\rangle}))= F\circ \omega_{\langle {\sf e}',e\rangle}\circ
\omega_{\langle {\sf e}'\rangle}^{\langle {\sf e}', e\rangle}(L_{\langle {\sf e}'\rangle})= F_{\langle {\sf e}'\rangle}(L_{\langle {\sf e}'\rangle})\subset (-1,1)$. Moreover, since the mapping $\omega_{\langle {\sf e}'\rangle}^{\langle {\sf e}', e\rangle}:X({\langle {\sf e}'\rangle})\to X({\langle {\sf e}',e\rangle})$ is a covering, its restriction $({\rm Im}F\circ \omega_{\langle {\sf e}'\rangle})^{-1}(0)\to ({\rm Im}F\circ \omega_{\langle {\sf e}', e\rangle})^{-1}(0)$  is a covering. Hence the image under $\omega_{\langle {\sf e}'\rangle}^{\langle {\sf e}', e\rangle}$ of a connected component of $({\rm Im}F\circ \omega_{\langle {\sf e}'\rangle})^{-1}(0)$
is open and closed in $({\rm Im}F\circ \omega_{\langle {\sf e}', e\rangle})^{-1}(0)$. Hence, $L_{\langle {\sf e}', e\rangle}= \omega_{\langle {\sf e}'\rangle}^{\langle {\sf e}', e\rangle}(L_{\langle {\sf e}'\rangle})$  is a connected component of the preimage of $(-1,1)$ under $ F_{\langle {\sf e}', e\rangle}$.
Put $q_{\langle {\sf e}', e\rangle}=\omega_{\langle {\sf e}'\rangle}^{\langle {\sf e}', e\rangle}(q_{\langle {\sf e}'\rangle})=\omega_{\langle {\sf e}'\rangle}^{\langle {\sf e}', e\rangle}\circ \omega^{\langle {\sf e}'\rangle}(\tilde{q})=\omega^{\langle {\sf e}', e\rangle}(\tilde{q})$. Note that
$F_{\langle {\sf e}', e\rangle}(q_{\langle {\sf e}', e\rangle})=F\circ\omega_{\langle {\sf e}', e\rangle}(q_{\langle {\sf e}', e\rangle})=
F(q)=q'$.

We prove now that
the free homotopy class $\reallywidehat{{\sf e}'_{\langle {\sf e}',e \rangle}}$
in $X(\langle {\sf e}',e \rangle)$
that is related to ${\sf e}'$  intersects $L_{\langle {\sf e}',e \rangle} $. Consider any loop $\gamma'_{\langle{\sf e}',e  \rangle}$ in $X(\langle {\sf e}',e \rangle)$ with some base point $q'_{\langle {\sf e}',e \rangle}$, that represents $\reallywidehat{{\sf e}'_{\langle {\sf e}',e \rangle  }}$.
There exists a loop ${\gamma}'_{\langle {\sf e}' \rangle}$ in $ X(\langle {\sf e}' \rangle)$ which represents $\reallywidehat{{\sf e}'_{\langle{\sf e}' \rangle}}$ such that $\omega_{\langle {\sf e}'\rangle}^{\langle {\sf e}',e\rangle  }({\gamma}'_{\langle {\sf e}'\rangle})=\gamma'_{\langle {\sf e}',e\rangle}$.
Such a curve ${\gamma}'_{\langle {\sf e}' \rangle}$ can be obtained as follows. There is a loop $\gamma''_{\langle {\sf e}',e\rangle}$ in $X(\langle {\sf e}',e\rangle  )$    with base point ${q}_{\langle {\sf e}',e\rangle}$ that represents $({\sf e}'  )_{\langle {\sf e}',e\rangle}$,  and a curve $\alpha'_{\langle {\sf e}',e\rangle}$ in $X(\langle{\sf e}',e\rangle)$
with initial point ${q}_{\langle {\sf e}',e\rangle}$ and terminal point  ${q}'_{\langle {\sf e}',e\rangle}$, such that $\gamma'_{\langle {\sf e}',e\rangle}$ is homotopic with fixed endpoints to $(\alpha'_{\langle {\sf e}',e\rangle})^{-1}\, \gamma''_{\langle{\sf e}',e\rangle}\,    \alpha'_{\langle {\sf e}',e\rangle}$. Consider the lift $\tilde{\gamma}''$ of $\gamma''_{\langle {\sf e}',e\rangle}$ to $\tilde X$ with initial point $\tilde{q}$, and the lift $\tilde{\alpha}'$ of $\alpha'_{\langle {\sf e}',e\rangle}$ with initial point $\tilde{q}$. The terminal point of $\tilde{\gamma}''_{\langle {\sf e}',e\rangle}$ equals $\sigma(\tilde{q})$ for the covering transformation $\sigma= ({\rm Is}^{\tilde{q}})^{-1}({\sf e}' )=
({\rm Is}^{\tilde{q}_0})^{-1}({\sf e}' _0)$. (See equation \eqref{eq1''}.)
The terminal point of the curve $(\tilde{\alpha}')^{-1} \tilde{\gamma}''_{\langle {\sf e}',e\rangle}\sigma(\tilde{\alpha}')$ is obtained from its initial point by applying the covering transformation $\sigma$. Hence,
$\omega^{\langle {\sf e}'\rangle}((\tilde{\alpha}')^{-1} \tilde{\gamma}''_{\langle {\sf e}',e\rangle}\sigma(\tilde{\alpha}'))$
is a closed curve in $X(\langle{\sf e}'  \rangle )$ that represents $\reallywidehat{{\sf e}' _{\langle {\sf e}'\rangle}}$ and projects to $(\alpha'_{\langle {\sf e}',e\rangle})^{-1}\, \gamma''_{\langle {\sf e}',e\rangle}\,    \alpha'_{\langle {\sf e}',e\rangle}$ under $\omega_{\langle {\sf e}' \rangle}^{\langle {\sf e}',e\rangle  }$.
Since $\gamma'_{\langle {\sf e}',e  \rangle}$ is homotopic to $(\alpha'_{\langle {\sf e}',e \rangle})^{-1}\, \gamma''_{\langle {\sf e}',e\rangle}\,    \alpha'_{\langle {\sf e}',e\rangle}$ with fixed base point, it also has a lift to $X(\langle {\sf e}'\rangle)$ which represents $\reallywidehat{{\sf e}' _{\langle {\sf e}'\rangle}}$.

Since $\reallywidehat{{\sf e}'_{\langle  {\sf e}'  \rangle}}$ intersects $L_{\langle {\sf e}' \rangle}$,
the loop
${\gamma}'_{\langle {\sf e}' \rangle}$ has an intersection point $p'_{\langle {\sf e}' \rangle}$ with $L_{\langle {\sf e}'\rangle}$. The point $p'_{\langle {\sf e}',e \rangle  }=  \omega_{\langle {\sf e}'\rangle}^{\langle {\sf e}',e\rangle  }(p'_{\langle {\sf e}'\rangle}    )$ is contained in $\gamma'_{\langle {\sf e}',e \rangle  }$ and in $L_{\langle {\sf e}',e\rangle}$.
We proved that the free homotopy class $\reallywidehat{{\sf e}'_{\langle {\sf e}',e\rangle } }$     in $X(\langle {\sf e}',e \rangle  )$ intersects $L_{\langle {\sf e}',e\rangle  }$.

\noindent {\bf 3.4. A system of generators associated to a standard bouquet of circles.}
We claim that
the system of generators  $\;{\sf e}'_{\langle {\sf e}',e\rangle }, \;\; e_{\langle {\sf e}',e\rangle }\;$ of $\;\pi_1(X(\langle {\sf e}',e\rangle),q_{\langle {\sf e}',e\rangle })$ is associated to a standard bouquet of circles for $X(\langle {\sf e}',e\rangle)$. If ${\sf e}'\in \mathcal{E}$ the claim can be obtained as in paragraph 3.2.

Suppose ${\sf e}'=e'e''$ for $e', e''\in \mathcal{E}$. Consider the system $\mathcal{E}'$ of generators of $\pi_1(X,q)$ that is obtained from $\mathcal{E}$ by replacing $e'$ by $e'e''$.
If $e'$ and $e''$ correspond to a handle of $X$, then $\mathcal{E}'$
is associated to a part of a standard bouquet of circles for $X$, see
Figure 4a for the case when $e'$ is represented by an $\alpha$-curve and $e''$ is represented by a $\beta$-curve. The situation when $e'$ is represented by a $\beta$-curve and $e''$ is represented by an $\alpha$-curve is similar. The claim is obtained as in paragraph 3.2. In this case $X(\langle {\sf e}',e\rangle)$ is a torus with a hole.

\begin{figure}[h]
\begin{center}
\includegraphics[width=10cm]{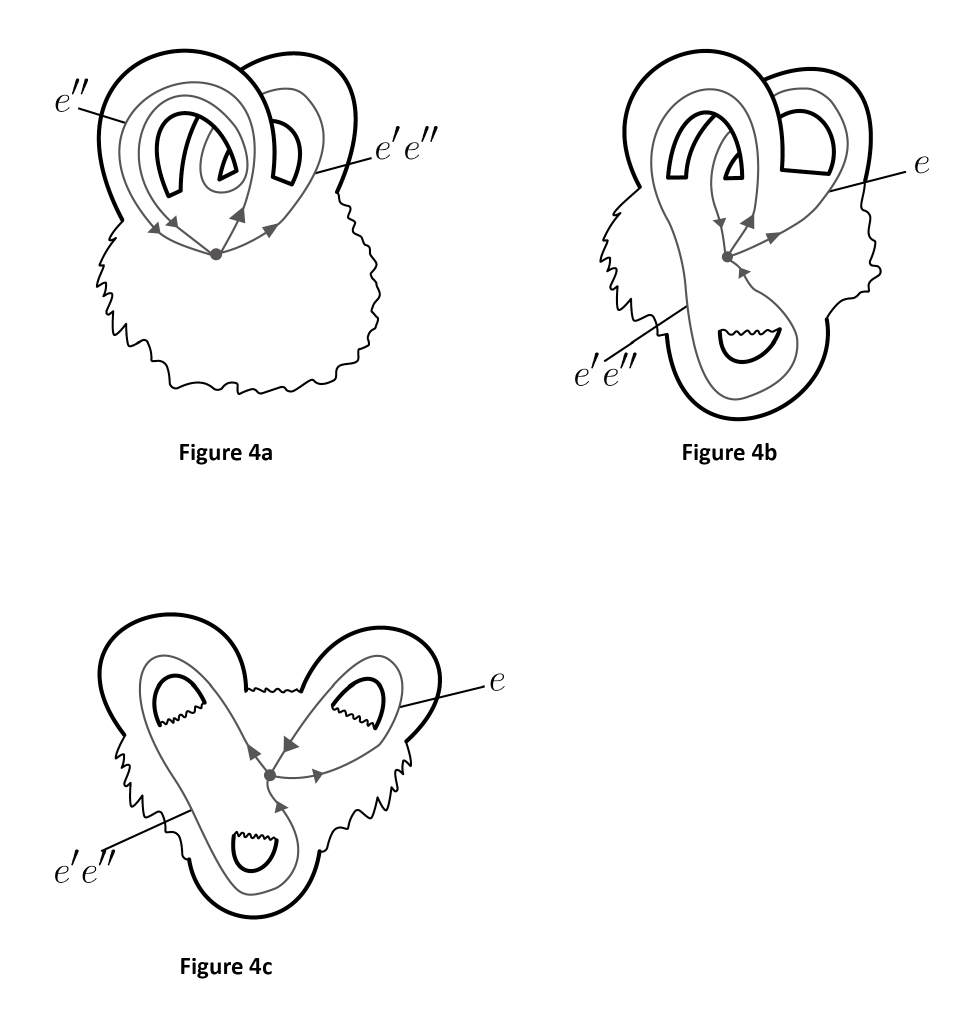}
\end{center}
\caption{Standard bouquets of circles}\label{fig4}
\end{figure}

Suppose one of the pairs $(e,e')$ or $(e,e'')$ corresponds to a handle of $X$.
We assume that $e$ corresponds to an $\alpha$-curve and $e'$ corresponds to a $\beta$-curve
of a handle of $X$
(see Figure 4b). The remaining cases are treated similarly, maybe, after replacing $e'e''$ by $e''e'$ (see paragraph 3.2). With our assumption $\mathcal{E}'$ is
not associated to a standard bouquet of circles for $X$. Nevertheless, the pair $(e_{\langle {\sf e}',e\rangle }, {\sf e}'_{\langle {\sf e}',e\rangle })$ with ${\sf e}'=e'e''$ is associated to a standard bouquet of circles for $X(\langle {\sf e}',e\rangle)$. This can be seen as before. Consider a (non-standard) bouquet of circles in $X$ corresponding to $\mathcal{E}'$ and take its union with a disc $D$ around $q$. Lift this set to $\tilde X$. We obtain
the union of a collection of arcs in $\tilde X$ with initial point $\tilde q$, with a collection of discs in $\tilde X$ around $\tilde q$ and around the terminal points of the arcs. Take the union of the arcs and the discs.  The image in $X(\langle {\sf e}',e\rangle)$ of this union under the projection $\omega^{\langle {\sf e}',e\rangle }$ is the union of the two loops $(\gamma_{e})_{\langle {\sf e}',e\rangle}\cup (\gamma_{{\sf e}'})_{\langle {\sf e}',e\rangle}$, the disc $D_{\langle {\sf e}',e\rangle }$
and a set, that is contractible to $D_{\langle {\sf e}',e\rangle }$. Looking at the intersection of the two loops with a small circle contained in  $D_{\langle {\sf e}',e\rangle }$ and surrounding $q_{\langle {\sf e}',e\rangle }$, we see as before that
$(\gamma_{e})_{\langle {\sf e}',e\rangle}\cup (\gamma_{{\sf e'}})_{\langle {\sf e}',e\rangle}$ is a standard bouquet of circles for $X(\langle {\sf e}',e\rangle)$. In this case $X(\langle {\sf e}',e\rangle)$ is a torus with a hole.

In the remaining case no pair of generators among $e$, $e'$, and $e''$ corresponds to a handle.
In this case again $\mathcal{E}'$ does not correspond to a standard bouquet of circles for $X$. But $\{e_{\langle {\sf e}',e\rangle } , (e'e'')_{\langle {\sf e}',e\rangle }\}$ (maybe, after changing $e'e''$ to $e''e'$)
corresponds to a standard bouquet of circles for $X(\langle {\sf e}',e\rangle )$. (See Figure 4c for the case when walking along a small circle around $q$ counterclockwise, we meet the incoming and outgoing rays of representatives of the three elements of $\mathcal{E}$ in the order $e,e',e''$. If the order is different the situation is similar, maybe, after replacing $e'e''$ by $e''e'$.) In this case $X(\langle {\sf e}',e\rangle )$ is a planar domain.

\noindent{\bf 3.5. End of the proof.}
Consider first the case when $X(\langle {\sf e}',e \rangle  )$ is a torus with a hole.
Since $\reallywidehat{{\sf e}'_{\langle {\sf e}',e \rangle }}$ intersects $L_{\langle {\sf e}',e \rangle }$, we see as in the proof when $X$ itself is a torus with a hole, that the curve $L_{\langle {\sf e}',e \rangle }$ cannot be contractible or contractible to the hole, and
the intersection number must be different from zero. Then the intersection number with $L_{\langle {\sf e}',e \rangle }$ of the free homotopy class of one of the choices $e_{\langle {\sf e}',e \rangle }^{\pm 1}$ or $({\sf e}'e)_{\langle {\sf e}',e \rangle }$, denoted by  ${\sf e}'''_{\langle {\sf e}',e \rangle }$, is not zero and has the same sign. By Lemma \ref{lemfin3} each of the
$(F_{\langle {\sf e}',e \rangle })_*({\sf e}'_{\langle {\sf e}',e \rangle })$ and $(F_{\langle {\sf e}',e \rangle })_*({\sf e}'''_{\langle {\sf e}',e \rangle })$ is the product of at most two elements of $\pi_1(\mathbb{C}\setminus\{-1,1\},q')$
with $\mathcal{L}_-$ not exceeding
\begin{equation}\label{eqfin11}
2\pi\lambda_{{\sf e}'_{\langle {\sf e}',e \rangle },{\sf e}'''_{\langle {\sf e}',e \rangle }}\leq 2\pi \lambda_5(X),
\end{equation}
since ${\sf e}'$ is the product of at most two elements of $\mathcal{E}\cup \mathcal{E}^{-1}$ and ${\sf e}'''$ is the product of at most three elements of $\mathcal{E}\cup \mathcal{E}^{-1}$.
The element $e$ is the product of at most two different elements among the ${\sf e}'$ and ${\sf e}'''$ or their inverses. Hence, the monodromy $F_*(e)=(F _{\langle {\sf e}',e \rangle })_*(e_{\langle {\sf e}',e \rangle })$ is the product of at most four elements with $\mathcal{L}_-$ not exceeding \eqref{eqfin11}.
Hence,
\begin{equation}\label{eqfin12}
F_*(e)\leq 8\pi \lambda_5(X)\,.
\end{equation}

Consider now the case when $X(\langle {\sf e}',e \rangle)$ equals $\mathbb{P}^1$ with three holes. Since ${\sf e}'_{\langle {\sf e}',e \rangle}$ and $e_{\langle {\sf e}',e \rangle}$ correspond to a standard bouquet of circles for $X(\langle {\sf e}',e \rangle)$, the curves representing ${\sf e}'_{\langle {\sf e}',e \rangle}$ surround counterclockwise one of the holes, denoted by  $\mathcal{C}'$ , and the curves representing $e_{\langle {\sf e}',e \rangle}$ surround counterclockwise another hole, denoted by  $\mathcal{C}''$.
After applying a M\"obius transformation we may assume that the remaining hole, denoted by $\mathcal{C}_{\infty}$, contains the point $\infty$.
There are several possibilities for the behaviour of the curve $L_{\langle {\sf e}',e \rangle}$.
Since $\reallywidehat{{\sf e}'_{\langle {\sf e}',e \rangle}}$ intersects $L_{\langle {\sf e}',e \rangle}$, the curve $L_{\langle {\sf e}',e \rangle}$ must have limit points on $\mathcal{C}'$. The first possibility is that $L_{\langle {\sf e}',e \rangle}$ has limit points on $\partial\mathcal{C}'$ and $\partial\mathcal{C}''$, the second possibility is, $L_{\langle {\sf e}',e \rangle}$
has limit points on $\mathcal{C}'$ and $\mathcal{C}_{\infty}$, the third possibility is, $L_{\langle {\sf e}',e \rangle}$ has all limit points on $\mathcal{C}'$, and $\mathcal{C}''$ is contained in the bounded connected component of $\mathbb{C}\setminus ( L_{\langle {\sf e}',e \rangle}\cup \mathcal{C}')$.

In the first case the free homotopy classes $\reallywidehat{{\sf e}'_{\langle {\sf e}',e \rangle}}$ and $\reallywidehat{e_{\langle {\sf e}',e \rangle}^{-1}}$ have positive intersection number with the suitably oriented curve $L_{\langle {\sf e}',e \rangle}$. In the second case the free homotopy classes $\reallywidehat{{\sf e}'_{\langle {\sf e}',e \rangle}}$ and $\reallywidehat{ ({{\sf e}'e})_{\langle {\sf e}',e \rangle} }$ have positive intersection number with the suitably oriented curve $L_{\langle {\sf e}',e \rangle}$. In the third case the free homotopy classes of ${\sf e}'_{\langle {\sf e}',e \rangle}$,  $ ({{\sf e}'}^2 e)_{\langle {\sf e}',e \rangle} $ and of their product intersect $L_{\langle {\sf e}',e \rangle}$. The first two cases were treated in paragraph 2 of this section.
The statement concerning the third case is proved as follows.

Any curve that is contained in the complement of $\mathcal{C}'\cup L_{\langle {\sf e}',e \rangle}    $ has either winding number zero around $\mathcal{C}'$ (as a curve in the complex plane $\mathbb{C}$), or its winding number around $\mathcal{C}'$ coincides with the winding number around $\mathcal{C}''$.
On the other hand the representatives of the free homotopy class of ${\sf e}'_{\langle {\sf e}',e \rangle} $
have winding number $1$ around $\mathcal{C}'$ and winding number $0$ around $\mathcal{C}''$.
The representatives of the free homotopy class of $({{\sf e}'}^2 e)_{\langle {\sf e}',e \rangle}$ have winding number $2$ around $\mathcal{C}'$, and  winding number $1$ around $\mathcal{C}''$. Hence,
${\sf e}'_{\langle {\sf e}',e \rangle} $ and
 $({{\sf e}'}^2 e)_{\langle {\sf e}',e \rangle}$ intersect $L_{\langle {\sf e}',e \rangle}$.
By the same argument the free homotopy class of the product of ${\sf e}'_{\langle {\sf e}',e \rangle}$ and $({{\sf e}'}^2 e)_{\langle {\sf e}',e \rangle}$ intersects $L_{\langle {\sf e}',e \rangle}$.

We let ${\sf e}'''_{\langle {\sf e}',e \rangle}$ be equal to $e_{\langle {\sf e}',e \rangle}^{-1}$ in the first case, equal to $ ({{\sf e}'e})_{\langle {\sf e}',e \rangle} $ in the second case, and equal to $({{\sf e}'}^2 e)_{\langle {\sf e}',e \rangle}$ in the third case. Then $e$ is the product of at most three elements among the $({\sf e}')^{\pm 1}$ and $({\sf e}''')^{\pm 1}$.

By Lemma \ref{lemfin3} each of the
$(F_{\langle {\sf e}',e \rangle })_*({\sf e}'_{\langle {\sf e}',e \rangle })$ and $(F_{\langle {\sf e}',e \rangle })_*({\sf e}'''_{\langle {\sf e}',e \rangle })$ is the product of at most two elements of $\pi_1(\mathbb{C}\setminus\{-1,1\},q')$
with $\mathcal{L}_-$ not exceeding
\begin{equation}\label{eqfin13}
2\pi\lambda_{{\sf e}'_{\langle {\sf e}',e \rangle },{\sf e}'''_{\langle {\sf e}',e \rangle }}\leq 2\pi \lambda_7(X),
\end{equation}
We used that ${\sf e}'$ is the product of at most two elements of $\mathcal{E}\cup \mathcal{E}^{-1}$, $e\in \mathcal{E}\cup \mathcal{E}^{-1}$ and ${\sf e}'''$ is the product of at most five elements of $\mathcal{E}\cup \mathcal{E}^{-1}$.
Since $e$ is the product of at most three elements among the $({\sf e}')^{\pm 1}$ and $({\sf e}''')^{\pm 1}$, the monodromy $F_*(e)=(F _{\langle {\sf e}',e \rangle })_*(e_{\langle {\sf e}',e \rangle })$ is the product of at most six elements with $\mathcal{L}_-$ not exceeding \eqref{eqfin13}.
Hence,
\begin{equation}\label{eqfin12'}
F_*(e)\leq 12\pi \lambda_7(X)\,.
\end{equation}
The proposition is proved.
\hfill $\Box$

\section{($\sf{g},\sf{m}$)-bundles over Riemann surfaces}\label{sec:fin3}

Theorem \ref{thmfin2} is a consequence of the following theorem on $(0,3)$-bundles with a section.

\begin{thm}\label{thmfin3}
Over a connected Riemann surface of genus $g$ with $m+1$ holes there are up to isotopy no more than $(15 \exp( 6 \pi \lambda_{10}(X)))^{6(2g+m)}$
irreducible holomorphic $(0,3)$-bundles
with a holomorphic section.
\end{thm}

Theorem \ref{thmfin1} (with a weaker estimate) is a consequence of Theorem \ref{thmfin3}. Indeed, consider holomorphic (smooth, respectively) bundles whose fiber over each point $x\in X$ equals $\mathbb{P}^1$ with set of distinguished points $\{-1,1,f(x),\infty\}$ for a function $f$ which depends holomorphically (smoothly, respectively) on the points $x\in X$ and does not take the values $-1$ and $1$. Then we are in the situation of Theorem \ref{thmfin3}. The mapping $f$ is reducible, iff the bundle is reducible (see also Lemma \ref{lemGrom1'}).

The relation between Theorems \ref{thmfin2} and  \ref{thmfin3} will be obtained by representing families of $(1,1)$-bundles as double branched coverings of special $(0,4)$-bundles (see Definition \ref{defnEl5}) and using
Proposition \ref{propEl.2}.

\smallskip

\noindent {\bf Preparation of the proof of Theorem \ref{thmfin3}.}
The proof of Theorem \ref{thmfin3} will go now along the same lines as the proof of Theorem \ref{thmfin1} with some modifications. Recall that the set
\begin{align*}
\mathcal{H}= & \{ \{z_1,z_2,z_3\} \in C_3(\mathbb{C})\diagup \mathcal{S}_3: \mbox{the three points}\; z_1,z_2, z_3  \nonumber\\
& \mbox{are contained in a real line in the complex plane}\}
\end{align*}
is a smooth real hypersurface of $C_3(\mathbb{C})\diagup \mathcal{S}_3$.

As before, for each complex affine self-mapping $M$ of the complex plane we consider the diagonal action $\;M\big((z_1,z_2,z_3)\big)=\big(M(z_1),M(z_2),M(z_3)\big)\;$ on points $\,(z_1,z_2,z_3)\in C_3(\mathbb{C})\,$, and the diagonal action $\;M\big(\{z_1,z_2,z_3\}\big)=\{M(z_1),M(z_2),M(z_3)\}\;$ on points $\{z_1,z_2,z_3\}\in C_3(\mathbb{C})\diagup \mathcal{S}_3$.

The following two lemmas replace Lemma \ref{lemfin2} in the case of
$(0,3)$-bundles with a section.

\begin{lemm}\label{lemfin4}
Let $A$ be an annulus with an orientation of simple closed dividing curves and $F:A \to C_3(\mathbb{C})\diagup \mathcal{S}_3$ a holomorphic mapping.
Suppose $L_A$ is a simple relatively closed curve in $A$ with limit points on both boundary circles of $A$, and $F(L_A) \subset \mathcal{H}$. Moreover, we assume that for a point $q_A\in L_A$ the value $F(q_A)$ is in the totally real subspace $ C_3(\mathbb{R})\diagup \mathcal{S}_3$.
Let $e_A\in \pi_1(A,q_A)$ be the positively oriented generator of the fundamental group of $A$ with base point $q_A$.
If the braid $b\stackrel{def}=F_*(e_A) \in \mathcal{B}_3 $ is different from $\sigma_j^k \, \Delta_3^{2 \ell'} $ with
$j$ equal to $1$ or $2$, and $k\neq 0$ and $\ell'$ being integers, then
\begin{equation}\label{eqfin14}
\mathcal{L}_-(\vartheta(b)) \leq 2\pi \lambda(A).
\end{equation}
\end{lemm}
Notice that the braids  $\sigma_j^k \, \Delta_3^{\ell} $ for odd $\ell$ are exceptional for Theorem \ref{thmbr.3}, but not exceptional for Lemma \ref{lemfin4}. The reason is that the braid in Lemma \ref{lemfin4} is related to a mapping of an annulus, not merely to a mapping of a rectangle.
For $t\in [0,\infty)$ we put
$\log_+ t\stackrel{def}=
\begin{cases} \log t& \; t \in [1,\infty) \\
0 & \; t \in [0,1)\;\;\;\; .\\
\end{cases}$

\begin{lemm}\label{lemfin3a}
If the braid in Lemma {\rm \ref{lemfin4}} equals $b = \sigma_j^k \, \sigma_{j'}^{k'} \,\Delta_3^{\ell}$ with $j$ and $j'$ equal to $1$ or to $2$,  $j'\neq j$, and $k$ and $k'$ being non-zero integers,
and $\ell$ an even integer, then
\begin{equation}\label{eqfin15}
\log_+(3[\frac{|k|}{2}]) + \log_+(3[\frac{|k'|}{2}])\leq {\pi} \lambda(A).
\end{equation}
\end{lemm}
\noindent As before, for a non-negative number $x$ we denote by $[x]$ the smallest integer not exceeding $x$.

\medskip

\noindent {\bf Proof of Lemma \ref{lemfin4}.}
By the same argument as in the proof of Lemma  \ref{lemfin2} we may assume that
the annulus $A$ has smooth boundary, the mapping
$F$ extends continuously to the closure $\overline{A}$, and the curve $L_A$ is a smooth (connected) curve in $\overline{A}$ whose endpoints are on different boundary components of $A$.
By Lemma \ref{lem10.30} the inequality
\begin{equation}\label{eqfin16}
\Lambda(b_{tr})\leq \lambda(A)\,
\end{equation}
holds.

For $b\neq \sigma_j^k \, \Delta_3^{\ell}$ with $j$ equal to $1$ or $2$, and $k\neq 0$ and $\ell$ being integers, the statement of Lemma \ref{lemfin4} follows from Theorem \ref{thmbr.3}  in the same way as Lemma \ref{lemfin2} follows from Theorem \ref{thm1}. For $b=\sigma_j^k \, \Delta_3^{\ell}$ with $k=0$ the statement is trivial since then $ \vartheta({\rm Id})={\rm Id}$ and $\mathcal{L}_-({\rm Id})=0$.

To obtain the statement in the remaining case $b = \sigma_j^k \, \Delta_3^{2 \ell' +1} $ with
$j$ equal to $1$ or $2$, and $k$ and $\ell'$ being integers, we use Lemma \ref{lemfin3a}.
Notice that $\sigma_1\, \Delta_3 = \Delta_3 \, \sigma_2$ and $\sigma_2\, \Delta_3 = \Delta_3 \, \sigma_1$. Hence, $b^2 = \sigma_j^k\, \sigma_{j'}^k \,\Delta_3^{4 \ell'+2}$ with $\sigma_j\,\neq  \sigma_{j'}$. Let $\omega_2:A^2\to A$ be the two-fold unbranched covering of $A$ by an annulus $A^2$. The equality $\lambda(A^2)=2 \lambda(A)$ holds. Let ${q}_{A^2}$ be a point in $\omega_2^{-1}(q_A)$, and let ${L}_{q_{A^2}}$ be the lift of $L_A$ to $A^2$ that contains ${q}_{A^2}$.
Denote by ${\gamma}_{A^2}$ the loop $\omega_2^{-1}(\gamma_A)$ with base point $q_{A^2}$.
Then $ {F}\circ \omega_2 \mid {\gamma}_{A^2}$ represents $b^2$ and $(b^2)_{tr}$.
Lemma \ref{lemfin3a} applied to $ \sigma_j^k\, \sigma_{j'}^k \,\Delta_3^{4 \ell'+2}$ gives the estimate $2\log_+(3[\frac{|k|}{2}]) \leq \pi \lambda(A^2)=2\pi\lambda(A)$.
Since  $\vartheta(b)=\sigma_j^{2[\frac{|k|}{2}]\mbox{sgn}(k)}$,
the inequality \eqref{eqfin14} follows.
The lemma is proved. \hfill $\Box$

\medskip

\noindent {\bf Proof of Lemma \ref{lemfin3a}.} Lemma \ref{lemm1} gives $\Lambda_{tr}( \sigma_j^k \, \sigma_{j'}^{k'} \,\Delta_3^{\ell})=\Lambda_{tr}( \sigma_j^k \, \sigma_{j'}^{k'} )$.
Along the lines of proof of Statement 1 of Lemma \ref{lemm15}
we see that for $k \cdot k'\neq 0$

\begin{align*}
\Lambda_{tr}( \sigma_j^k \, \sigma_{j'}^{k'} )\geq \Lambda \big( _{tr}(\sigma_j^k)_{pb}\big) \,+\, \Lambda \big( _{pb}(\sigma_{j'}^{k'})_{tr}\big)\,.
\end{align*}
A similar argument as used in the proof of Theorem  \ref{thmbr.3}
shows that for any positive integer number $\ell$ and each of the $\sigma_j$ the inequality
$\Lambda \big( _{tr}(\sigma_j^{\pm\ell})_{pb}\big)\geq \Lambda \big( _{tr}(\sigma_j^{\pm 2[{\frac{\ell}{2}}]})_{pb}\big)$ holds. Theorem
\ref{thm1} implies
$
\Lambda \big( _{tr}(\sigma_j^{\pm 2[{\frac{\ell}{2}}]})_{pb}\big)
\geq \mathcal{L}_-(\sigma_j^{\pm 2[{\frac{\ell}{2}}]})= \log_+(3[\frac{\ell}{2}])\,.$
Since by \eqref{eqfin16} the inequality $\Lambda_{tr}(\sigma_j^k \, \sigma_{j'}^{k'} \,\Delta_3^{\ell}) \leq \lambda(A)$ holds,
the lemma is proved. \hfill $\Box$
\medskip

We want to emphasize that periodic $3$-braids are not of the form
$\sigma_j^k \,\Delta_3^{2\ell}$ for  non-zero powers $\sigma_j^k$ of a $\sigma_j$ and an integer number $\ell$, hence, the lemma is true also for periodic braids. For each periodic braid $b$ of the form $\sigma_1 \sigma_2= \sigma_1 ^{-1} \, \Delta_3$, $(\sigma_1 \sigma_2)^2=\sigma_1 \, \Delta_3$, $\sigma_2 \sigma_1= \sigma_2 ^{-1} \, \Delta_3$, $(\sigma_2 \sigma_1)^2=\sigma_2 \, \Delta_3$,
and $\Delta_3$ the $\mathcal{L}_-(\vartheta(b))$ vanishes. However, for instance for the conjugate $\sigma_1^{-2k} \Delta_3 \sigma_1^{2k}= \sigma_1^{-2k} \sigma_2^{2k} \Delta_3$ of $\Delta_3$ we have $\mathcal{L}_-(\vartheta(\sigma_1^{-2k} \Delta_3 \sigma_1^{2k}))= 2 \log(3|k|)$. Another example, for the conjugate
$\sigma_2^{-2k}\, \sigma_1 \sigma_2\,  \sigma_2^{2k}$ of $\sigma_1 \sigma_2$ we have
\begin{align}\nonumber
\sigma_2^{-2k}\, \sigma_1 \sigma_2\,  \sigma_2^{2k}
=  \sigma_2^{-2k-1}\, \Delta_3\,  \sigma_2^{2k} = \sigma_2^{-2k-1}\, \sigma_1^{2k} \,\Delta_3\,.
\end{align}
and  $\mathcal{L}_-(\vartheta(\sigma_2^{-2k}\, \sigma_1 \sigma_2\,  \sigma_2^{2k})) $ equals  $2 \log(3|k|)$.

Notice that the lemmas and Theorem \ref{thmbr.3} descend to statements on elements of $\mathcal{B}_3 \diagup \mathcal{Z}_3$ rather than on braids.
For an element $\textsf{b}$ of the quotient $\mathcal{B}_3 \diagup \mathcal{Z}_3$ we put $\vartheta(\textsf{b})= \vartheta(b)$ for any representative $b \in \mathcal{B}_3$ of $\textsf{b}$.

\smallskip

Lemma \ref{lemfin8a} below is an analog of Lemma \ref{lemfin3}. It will follow from Lemma \ref{lemfin4} in the same way as Lemma \ref{lemfin3} follows from Lemma \ref{lemfin2}.

\begin{lemm}\label{lemfin8a}
Let $X$ be a connected finite open Riemann surface, and let $F:X \to C_3(\mathbb{C}) \diagup \mathcal{S}_3$ be a non-contractible holomorphic map that is transverse to the hypersurface $\mathcal{H}$ in $ C_3(\mathbb{C}) \diagup \mathcal{S}_3$.
Suppose $L_0$ is a simple relatively closed curve in $X$
such that $F(L_0)$ is contained in $\mathcal{H}$, and for a point $q \in L_0$ the point $F(q)$ is contained in the totally real space $C_3(\mathbb{R}) \diagup \mathcal{S}_3$.
Let $e^{(1)}$ and $e^{(2)}$ be primitive elements of $\pi_1(X,q)$. Suppose that for $e=e^{(1)}$, $e=e^{(2)}$, and $e=e^{(1)}e^{(2)}$ the free homotopy class $\widehat e$ intersects $L_0$.
Then either the two monodromies of $F$ modulo the center  $F_*(e^{(j)})\diagup \mathcal{Z}_3,\, j=1,2,\,$ are powers
of the same element $\sigma_j\diagup \mathcal{Z}_3$
of $\mathcal{B}_3\diagup \mathcal{Z}_3$, or each of them is the product of at most two elements $\textsf{b}_1$ and $\textsf{b}_2$ of $\mathcal{B}_3\diagup \mathcal{Z}_3$ with
\begin{equation}\label{eqfin18}
\mathcal{L}_-(\vartheta(\textsf{b}_j)) \leq  2\pi \lambda_{e^{(1)},e^{(2)}},\, j=1,2,
\end{equation}
where
\begin{equation}\nonumber
\lambda_{e^{(1)},e^{(2)}} \stackrel{def}=\max\{\lambda(A(\reallywidehat{e^{(1)}})),\, \lambda(A(\reallywidehat{e^{(2)}})),\, \lambda(A(\reallywidehat{e^{(1)}\,e^{(2)}}))\}.
\end{equation}
\end{lemm}

\noindent{\bf Proof.}
Suppose for an element $e\in \pi_1(X,q)$ the free homotopy class $\widehat e$ intersects $L_0$.
We claim that all components of ${\sf P}^{-1}(L_0)$ have limit points on the boundary of ${\tilde X}\cong \mathbb{C}_+$.

Suppose the claim is not true. Then some connected component $\tilde{L}_0$ of  ${\sf P}^{-1}(L_0)$ is compact. Lift the holomorphic mapping $F\circ{\sf P}:\tilde{X}\cong \mathbb{C}_+\to C_3(\mathbb{C})\diagup\mathcal{S}_3$ to a holomorphic mapping $\widetilde{F\circ{\sf P}}:{\tilde X}\cong \mathbb{C}_+ \to C_3(\mathbb{C})$. For each $z\in
\mathbb{C}_+$ we let $\mathfrak{A}_z$ be the complex affine mapping that takes the first coordinate of  $\widetilde{F\circ{\sf P}}(z)$
to $-1$, the third to $1$, and the second to a point $g(z)\in \mathbb{C}\setminus\{-1,1\}$. Then $\mathfrak{A}_z(\widetilde{F\circ{\sf P}}(z))=\big(-1,g(z),1\big), z\in \mathbb{C}_+$. By the conditions of the lemma $g(z)$ is real on $\tilde{L}_0$. Then, since $g$ is holomorphic on $\mathbb{C}_+$, $g\equiv c$ on
$\mathbb{C}_+$ for a constant $c$ (see the proof of Lemma  \ref{lemfin1}). It follows that
$\widetilde{F\circ{\sf P}}(z)=\mathfrak{A}_z^{-1}(c), \,z\in  \mathbb{C}_+$, which means that
the image of $F\circ{\sf P}$ is contained in $\mathcal{H}$. This contradicts the fact that $F$ is not constant and is transverse to $\mathcal{H}$. The claim is proved.

It follows along the lines of the proof of Lemma \ref{lemfin1} that  an analog of Lemma \ref{lemfin1} holds. More precisely,  for any point
$\tilde q$ in ${\sf P}^{-1}(q)$, with
$A\stackrel{def}=\tilde{X}\diagup ({\rm Is}^{\tilde{q}})^{-1}(\langle e \rangle)$,  $\omega_{A}\stackrel{def}=\omega_{\langle e\rangle,\tilde{q}}$, and
$q_{A}\stackrel{def}=\omega^{\langle e\rangle,\tilde{q}}(\tilde{q})$
the connected component $L_{A}\subset A$ of $(\omega_{A})^{-1}(L_0)$ that contains $q_{A}$ has limit points on both boundary circles
of $A$.

Put $F_A=F\circ\omega_A$. By the conditions of Lemma \ref{lemfin8a} $F_A(L_A)=F(L_0)\subset \mathcal{H}$ and $F_A(q_A)\in C_3(\mathbb{R})\diagup\mathcal{S}_3$.
Let $e_A$ be the generator of $\pi_1(A,q_A)$ for which $\omega_A(e_A)=e$.
The mapping $F_A:A\to C_3(\mathbb{C})\diagup \mathcal{S}_3$, the point $q_A$ and the curve $L_A$  satisfy the conditions of Lemma \ref{lemfin4}.
Notice that the equality $(F_A)_*(e_A)= F_*(e)$ holds.
Hence, if $F_*(e)$ is not equal to $\sigma^k_j \Delta^{2\ell},$ $j=1$ or $2$, $k\neq 0, \ell\in \mathbb{Z}$, then inequality \eqref{eqfin14} holds for $F_*(e)$.

Suppose the two monodromies modulo center $F_*(e^{(j)})\diagup \mathcal{Z}_3,\, j=1,2,\,$ of Lemma \ref{lemfin8a}
are not (trivial or non-trivial)
powers of the same element $\sigma_j \diagup \mathcal{Z}_3$ of $\mathcal{B}_3\diagup \mathcal{Z}_3$.
Then at most two of the elements, $\;\,F_*(e^{(1)})\diagup \mathcal{Z}_3$,  $F_*(e^{(2)})\diagup \mathcal{Z}_3\;\,$, and $\;\;F_*(e^{(1)}e^{(2)})\diagup \mathcal{Z}_3=F_*(e^{(1)})\diagup \mathcal{Z}_3\cdot F_*(e^{(2)})\diagup \mathcal{Z}_3\; $, are powers of an element of the form $\sigma_j\diagup \mathcal{Z}_3$.

If the monodromies modulo center along two elements among $e^{(1)}$,  $e^{(2)}$, and $e^{(1)}e^{(2)}$ are
not (zero or non-zero) powers of a $\sigma_j\diagup \mathcal{Z}_3$ then by Lemma \ref{lemfin4} for each of these two monodromies modulo center inequality \eqref{eqfin18} holds, and the third monodromy modulo center is the product of two elements of $\mathcal{B}_3\diagup \mathcal{Z}_3$ for which  inequality \eqref{eqfin18} holds.
If the monodromies modulo center along two elements among $e^{(1)}$, $e^{(2)}$, and $e^{(1)}e^{(2)}$ have the form $\sigma_{j}^{k} \diagup \mathcal{Z}_3$ and  $\sigma_{j'}^{k'}\diagup\mathcal{Z}_3$,
then the $\sigma_j$ and the $\sigma_{j'}$ are different and $k$ and $k'$ are non-zero.
The third monodromy modulo center  has the form $\sigma_{j }^{\pm k} \sigma_{j'}^{\pm k'} \diagup \mathcal{Z}_3$ (or the order of the two factors interchanged).
Lemma \ref{lemfin3a} gives the inequality $\log_+(3[\frac{|k|}{2}]) + \log_+(3[\frac{|k'|}{2}])\leq \pi \lambda_{e^{(1)},e^{(2)}}$.
Since $\mathcal{L}_-(\vartheta(\sigma_{j}^{\pm k}))= \log_+(3[\frac{k}{2}])$ and  $\mathcal{L}_-(\vartheta(\sigma_{j'}^{\pm k'})) =\log_+(3[\frac{k'}{2}])$,
inequality \eqref{eqfin18} follows for the other two monodromies.
The lemma is proved.
\hfill $\Box$

\smallskip
The following lemma is more comprehensive than Lemma \ref{lemm3-braids2}.
\begin{lemm}\label{lemfin7}
Let $X$ be a connected finite open Riemann surface, and $F:X\to C_3(\mathbb{C})\diagup \mathcal{S}_3$ a smooth mapping. Suppose for a base point $q_1$ of $X$ each element of $\pi_1(X,q_1)$ can be represented by a curve with base point $q_1$ whose image under $F$ avoids $\mathcal{H}$. Then all monodromies of $F$ are powers of the same periodic braid of period $3$.
\end{lemm}

\noindent {\bf Proof.} By Lemma \ref{lemm3-braids2}
the monodromy of $F$ along each element of $\pi_1(X,q_1)$  is the power of a periodic braid with period $3$.
There is a smooth homotopy $F_s,\, s\in[0,1],$ of $F$, such that $F_0=F$, each $F_s$ is different from $F$ only on a small neighbourhood of $q_1$, each $F_t$ avoids $\mathcal{H}$ on this neighbourhood of $q_1$,
and
$F_1(q_1)$ is the set of vertices of an equilateral triangle with barycenter $0$.
Since $F$ and $F_1$ are free homotopic, their monodromy homomorphisms are conjugate, and it is enough to prove the statement of the lemma for $F_1$.

For notational convenience we will keep the notation $F$ for the new mapping and assume that $F(q_1)$ is the set of vertices of an equilateral triangle with barycenter $0$.
The monodromy $F_*(e)$ along each element $e\in \pi_1(X,q_1)$ is a power of a periodic braid of period $3$. Hence, $\tau_3(F_*(e))$ is a cyclic permutation. Consider the braid $b$ with base point $F(0)$ that corresponds to rotation by the angle $\frac{2\pi}{3}$, i.e. it is represented by the geometric braid $t\to e^{\frac{i2\pi t}{3}}F(0),\, t\in [0,1],$ that avoids $\mathcal{H}$.
There exists an integer $k$ such that $F_*(e)\, b^k$ is a pure braid that is represented
by a mapping that avoids $\mathcal{H}$.
Hence, $F_*(e)\, b^k$ represents $\Delta_3^{2l}$ for some integer $l$. We proved that for each $e\in \pi_1(X,q_1)$ the monodromy  $F_*(e)$ is represented by rotation of $F(0)$ around the origin by the angle $\frac{2\pi j}{3}$ for some integer $j$.
The Lemma is proved. \hfill $\Box$

\medskip

Let as before $X$ be a finite open connected Riemann surface. The following proposition is the main ingredient of the proof of Theorem \ref{thmfin3}. As before $\mathcal{E}\subset \pi_1(X,q_0)$ denotes a standard system of generators of the fundamental group with base point $q_0\in X$ that is associated to a standard bouquet of circles for $X$.

\begin{prop}\label{propfin4}
Let $(X\times \mathbb{P}^1,{\rm pr}_1,\mathbold{E},X)$
be an
irreducible holomorphic special $(0,4)$-bundle
over a finite open Riemann surface $X$, that is not isotopic to a locally holomorphically trivial bundle.  Let $F(x),\, x \in X,$ be the set of finite distinguished points in the fiber over $x$. Assume that $F$ is transverse to $\mathcal{H}$.
Then there exists a complex affine mapping $M$ and a point $q\in X$ such that $M\circ
F(q)$ is contained in $C_3(\mathbb{R})\diagup \mathcal{S}_3$, and for an arc $\alpha$ in $X$ with initial point $q_0$ and terminal point $q$ and each element $e_j\in \mbox{Is}_{\alpha}(\mathcal{E})$ the monodromy modulo center $(M\circ F)_*(e_j)\diagup \mathcal{Z}_3$ can be written as product of at most $6$ elements  $\textsf{b}_{j,k},\, k=1,2,3,4,5,6,$ of $\mathcal{B}_3 \diagup \mathcal{Z}_3$ with
\begin{equation}\label{eqfin7}
\mathcal{L}_-(\vartheta(\textsf{b}_{j,k})) \leq 2\pi \lambda_{10}(X).
\end{equation}
If $X$ is a torus with a hole the monodromy along each $e_j$ is the product of at most $4$ elements with
$\mathcal{L}_-(\vartheta(\textsf{b}_{j,k})) \leq 2\pi \lambda_3(X)$, and in case of a planar domain the monodromy along each $e_j$ is the product of at most $6$ elements with
$\mathcal{L}_-(\vartheta(\textsf{b}_{j,k})) \leq 2\pi \lambda_8(X)$.\\
If $X$ is the sphere with $m=3$ holes,  the monodromy along each $e_j$ is the product of at most $6$ elements with
$\mathcal{L}_-(\vartheta(\textsf{b}_{j,k})) \leq 2\pi \lambda_5(X)$.

\end{prop}

\medskip

\noindent {\bf Proof of Proposition \ref{propfin4}.} Since the bundle is not isotopic to a locally holomorphically trivial bundle, it is not possible that
all monodromies are powers of the same periodic braid, and by Lemma \ref{lemfin7} the set
\begin{align}\label{eqfin7a}
L\stackrel{def}= & \{ z \in X: F(z) \in\mathcal{H }\}
\end{align}
is not empty.

\noindent {\bf 1. A torus with a hole.} Let $X$ be a torus with a hole
and let  $\mathcal{E}=\{e'_0 ,\,e''_0\}$ be a set of  generators of $\pi_1(X,q_0)$ that is associated to a standard bouquet of circles for $X$.
There exists a connected component $L_0$ of $L$ which is not contractible and not contractible to the hole. Indeed, otherwise there would be a base point $q_1$ and a curve $\alpha_{q_1}$ that joins $q_0$ with $q_1$, such that for both elements of
$\mbox{Is}_{\alpha_{q_1}}(\mathcal{E})$
there would be representing loops with base point $q_1$ which do not meet $L$, and hence, by Lemma \ref{lemfin7} the monodromies along both elements would be powers of a single periodic braid of period $3$.
Hence, as in the proof of Proposition \ref{propfin2a} there exists a component $L_0$ of $L$, which is a simple smooth relatively closed curve in $X$, such that
the free homotopy class of one of the elements of $\mathcal{E}$, say of $e'_0$,
has positive intersection number with $L_0$ after orienting $L_0$ suitably.
Put ${\sf e}'_0 =e'_0$.
Moreover, the intersection number with $L_0$ is positive for the free homotopy class of one of the elements
${e_0''}^{\pm 1}$ or $e'_0 e''_0$.
Denote this element by ${\sf e}_0''$. (Since $\reallywidehat{e'_0 e''_0}=\reallywidehat{e''_0 e'_0}$ we may also put ${\sf e}_0''=e''_0 e'_0$ if the free homotopy class of $e'_0 e''_0$ intersects $L_0$.)
Put $\mathcal{E}_2'= \{{\sf e}_0', {\sf e}_0''\}$. The free homotopy class of each element of  $\mathcal{E}_2'$ and of the product of its two elements
intersects $L_0$.
We proved the following claim which we formulate
for later use.
\smallskip

\noindent {\bf Claim 12.1} {\it Let $X$ be a torus with a hole and $\mathcal{E}=\{e'_0 ,\,e''_0\}$ a set of standard generators of $\pi_1(X,q_0)$. There exists a component $L_0$ of $L$, which is a simple smooth relatively closed curve in $X$, and two primitive elements ${\sf e}_0'$ and
${\sf e}_0''$ in $\langle \mathcal{E} \rangle$ such that the free homotopy class of each element of  $\mathcal{E}_2'' \stackrel{def}= \{{\sf e}_0', {\sf e}_0''\} $ and of the product of its two elements intersects $L_0$. Moreover, one of the ${\sf e}_0'$ and ${\sf e}_0''$ is an element of $\mathcal{E}$, the other is in $\mathcal{E}\cup \mathcal{E}^{-1}$ or is the product of two
elements of $\mathcal{E}$. Each element of  $\mathcal{E}$ is the product of at most two elements of $\mathcal{E}'_2\cup {\mathcal{E}'_2}^{-1}$.}

\noindent The proof of the proposition in the case of a torus with a hole is finished as follows.
Move the base point $q_0$ to a point $q \in L_0$ along a curve $\alpha$, and consider the respective generators ${\sf e}'=\mbox{Is}_{\alpha}({\sf e}'_0)$ and ${\sf e}''=\mbox{Is}_{\alpha}({\sf e}''_0)$ of the fundamental group $\pi_1(X,q)$ with base point $q$. Since $F(L_0)\subset \mathcal{H}$ there is a complex affine mapping $M$ such that $M\circ F(q)\in C_3(\mathbb{R})\diagup \mathcal{S}_3$.
Since $F$ is irreducible, the monodromy maps modulo center $(M\circ F)_*({\sf e}')\diagup \mathcal{Z}_3$ and $(M\circ F)_*({\sf e}'')\diagup \mathcal{Z}_3$ are not powers of a single standard generator $\sigma_j\diagup \mathcal{Z}_3$  of $\mathcal{B}_3 \diagup \mathcal{Z}_3$ (see Lemma \ref{lemEl.0}, or
Lemma 7 of \cite{Jo5}).
Hence, the second option of Lemma \ref{lemfin8a} occurs. We obtain that each of the $(M\circ F)_*({\sf e}')\diagup \mathcal{Z}_3$ and $(M\circ F)_*({\sf e}'')\diagup \mathcal{Z}_3$
is a product of at most two elements ${\sf{b}}_j$ of $\mathcal{B}_3\diagup \mathcal{S}_3$ with $\mathcal{L}_-(\vartheta({\sf{b}}_j))\leq 2\pi\lambda_3(X)$. Hence, $(M\circ F)_*(e')\diagup \mathcal{Z}_3$ and $(M\circ F)_*(e'')\diagup \mathcal{Z}_3$ are products of at most $4$
elements of $\mathcal{B}_3\diagup \mathcal{Z}_3$ with this property.
The proposition is proved for tori with a hole.

\noindent {\bf  2. A planar domain.} Let $X$ be a planar domain. Maybe, after applying a M\"obius transformation, we represent $X$ as the Riemann sphere with holes $\mathcal{C}_j$, $j=1,\ldots,m+1,$ such that $\mathcal{C}_{m+1}$ contains $\infty$. Recall, that we have chosen a standard system $\mathcal{E}$ of generators $e_{j,0}, \, j=1,\ldots,m,$ of the fundamental group $\pi_1(X,q_0)$
with base point $q_0$ that is associated to a standard bouquet of circle for $X$. Hence, $e_{j,0}$ is represented by a loop with base point $q_0$ that surrounds $\mathcal{C}_j$ counterclockwise and does not surround any other hole.

We claim that there is a connected component $L_0$ of $L$
of one of the following kinds.
Either $L_0$ has limit points on the boundary of two different holes (one of them may contain $\infty$) (first kind),
or a component $L_0$ has limit points on a single hole $\mathcal{C}_j,\, j\leq m+1,$
and $\mathcal{C}_j\cup L_0$ divides the plane $\mathbb{C}$ into two connected components each of which contains a hole (maybe, only the hole containing $\infty$) (second kind),
or there is a compact component $L_0$ that divides $\mathbb{C}$ into two connected components each of which contains at least two holes (one of them may contain $\infty$).
Indeed, suppose each non-compact component of $L$ has boundary points on the boundary of a single hole and the union of the component with the hole does not separate the remaining holes of $X$, and for each compact component of $L$ one of the connected components of its complement in $X$ contains at most one hole.
Then there exists a base point $q_1$, a curve $\alpha_{q_1}$ in $X$ with initial point $q_0$ and terminal point $q_1$, and a representative of each element of  $\mbox{Is}_{\alpha_{q_1}}(\mathcal{E})\subset\pi_1(X,q_1)$ that avoids $L$.
Lemma \ref{lemfin7} implies that all monodromies modulo center are powers of a single periodic element of $\mathcal{B}_3\diagup \mathcal{Z}_3$ which is a contradiction.

If there is a component $L_0$ of the first kind we choose the same set of primitive elements $\mathcal{E}_2'\subset \mathcal{E}_2 \subset \pi_1(X,q_0)$ as in the proof of
Proposition \ref{propfin2a}
in the planar case.
The free homotopy class of each element of $\mathcal{E}_2'$ and of the product of two different elements of $\mathcal{E}_2'$ intersects $L_0$.
Moreover, each element of $\mathcal{E}$ is the product of at most two elements of $\mathcal{E}_2'$.

Suppose there is no component of the first kind but a component $L_0$ of the second kind. Assume first that all limit points of $L_0$ are on the boundary of a hole $\mathcal{C}_j$ that does not contain $\infty$. Put $\mathcal{E}_3'=\{e_{j,0}\} \cup_{1\leq k\leq m,\, k\neq j}\{ e_{j,0}^2 e_{k,0}\}$. Each element of $\mathcal{E}_3'$ is a primitive element and is the product of at most three generators contained in the set $\mathcal{E}$.
Further, each element of $\mathcal{E}$ is the product of at most three elements of $\mathcal{E}_3'\cup {\mathcal{E}_3'}^{-1}$.

The free homotopy class of each element of $\mathcal{E}_3'$ and of each product of two different elements of $\mathcal{E}_3'$
intersects $L_0$. Indeed,
any curve that is contained in
$\mathbb{C}\setminus \big(\mathcal{C}_j\cup L_0\big)$ has either winding number zero around $\mathcal{C}_j$ (as a curve in the complex plane $\mathbb{C}$), or its winding number around $\mathcal{C}_j$ coincides with the winding number around each of the holes in the bounded connected component of
$\mathbb{C}\setminus \big(\mathcal{C}_j\cup L_0\big)$.
On the other hand the representatives of the free homotopy class of $e_{j,0}$
have winding number $1$ around $\mathcal{C}_j$ and winding number $0$ around each other hole that does not contain $\infty$. The representatives of the free homotopy class of $e_{j,0}^2 e_{k,0}$, $k\leq m,\, k\neq j$,
have winding number $2$ around $\mathcal{C}_j$,  winding number $1$ around $\mathcal{C}_k$, and winding number zero around each other hole $\mathcal{C}_l,\, l\leq m$.
The argument for products of two elements of $\mathcal{E}_3'$ is the same.

Assume that the limit points of $L_0$ are on the boundary of the hole $\mathcal{C}_{\infty}$ that contains $\infty$.
Let $\mathcal{C}_{j_0}$ and $\mathcal{C}_{k_0}$ be holes that are contained in different components of $\mathbb{C} \setminus (L_0 \cup \mathcal{C}_{\infty})$, and let $e _{j_0,0}$ and $e _{k_0,0}$ be the elements of $\mathcal{E}$ whose representatives surround $\mathcal{C}_{j_0}$, and $\mathcal{C}_{k_0}$ respectively.
Denote by  $\mathcal{E}'_3$ the set that consists of the elements  $e _{j_0,0}e _{k_0,0}\,$,  $\;e _{j_0,0}^2e _{k_0,0}\,$,
and all elements $ e _{j_0,0}e _{k_0,0} \tilde {e}_0$ with $\tilde{e}_0$ running over $ \mathcal{E}\setminus\{ e _{j_0,0},e _{k_0,0}\}$.
Each element of $\mathcal{E}_3'$ is the product of at most $3$ elements of $\mathcal{E}$, and each element of $\mathcal{E}$ is the product of at most $3$ elements of $\mathcal{E}'_3\cup (\mathcal{E}'_3)^{-1}$.

Each element of $\mathcal{E}'_3$ and each product of at most two  different elements of $\mathcal{E}'_3$ intersects $L_0$. Indeed, if a closed curve is contained in one of the components of $\mathbb{C}\setminus (L_0\cup \mathcal{C}_{\infty})$ then its winding number around each hole contained in the other component is zero. But for all mentioned elements there is a hole in each component of  $\mathbb{C}\setminus (L_0 \cup \mathcal{C}_{\infty})$ such that the winding number of the free homotopy class  of the element around the hole does not vanish.

Notice that in case of $m+1=3$ holes only these two possibilities for the curve $L_0$ may occur. We proved the following claim which we formulate for later use.
\smallskip

\noindent {\bf Claim 12.2} {\it If $X$ equals $\mathbb{P}^1$ with  $m+1=3$ hole and $\mathcal{E}$ is a standard system of generators of  $\pi_1(X,q_0)$, then there is a set $\mathcal{E}_3'=\{{\sf e}'_0,{\sf e}''_0\} \subset \pi_1(X,q_0)$, such that one of the elements of $\mathcal{E}_3'$ is the product of at most two elements of $\mathcal{E}\cup \mathcal{E}^{-1}$, and the free homotopy classes of both elements and of their product intersect $L_0$.
Moreover, $e$ and $e'$ are products of at most three factors, each an element of  $\mathcal{E}_3'\cup \mathcal{E}_3'^{-1}$. }

\smallskip

\noindent For the general case of planar domains we also need the following considerations.
Suppose there are no components of $L$ of the first or the second kind, but there is a connected component $L_0$ of $L$ of the third kind.
Let $\mathcal{C}_{j_0}$ be a hole contained in the bounded component of the complement of $L_0$ in $\mathbb{C}$, and let $\mathcal{C}_{k_0},\, k_0\leq m,$ be a hole that is contained in the unbounded component of $\mathbb{C} \setminus L_0$. Let $e _{j_0,0}$ and $e _{k_0,0}$ be the elements of $\mathcal{E}$ whose representatives surround $\mathcal{C}_{j_0}$, and $\mathcal{C}_{k_0}$ respectively.
Consider the set $\mathcal{E}'_4$ consisting of the following elements: $e_{j_0,0} e_{k_0,0}$,
$e_{j_0,0}^2 e_{k_0,0}$, and $e_{j_0,0}^2 e_{k_0,0} \tilde{ e}_0$ for each $\tilde{e}_0\in \mathcal{E}$ different from $e_{j_0,0}$ and $e_{k_0,0}$. Each element of $\mathcal{E}'_4$ is the product of at most $4$ elements of $\mathcal{E}$ and each element of $\mathcal{E}$ is the product of at most $3$ elements of $\mathcal{E}_4'\cup (\mathcal{E}_4')^{-1}$. The product of two different elements of $\mathcal{E}'_4$ is contained in $\mathcal{E}'_8$.

The free homotopy classes of each element of $\mathcal{E}'_4$ and of each product of two different elements of $\mathcal{E}'_4$ intersect $L_0$.
Indeed, if a loop is contained in the bounded connected component of $\mathbb{C} \setminus L_0$, its winding number around the holes $\mathcal{C}_j\,, j\leq m,$ contained in the unbounded component is zero. If a loop is contained in the unbounded connected
component of $\mathbb{C} \setminus L_0$, its winding numbers around all holes contained in the bounded connected component are equal. But the winding numbers of $e_{j_0,0} e_{k_0,0}$ and $e_{j_0,0}^2 e_{k_0,0}$ around the hole $\mathcal{C}_{j_0}$ are positive and the winding numbers around the other holes that are contained in the bounded connected component of $\mathbb{C} \setminus L_0$ vanish, hence the representatives of these two elements cannot be contained in the unbounded  component of $\mathbb{C} \setminus L_0$. Since the winding numbers of representatives of these elements around $\mathcal{C}_ {k_0}$ are positive, the representatives cannot be contained in the bounded component of $\mathbb{C} \setminus L_0$.
For representatives of  the elements $e_{j_0,0}^2 e_{k_0,0} \tilde{e}_0$
the winding numbers around $\mathcal{C}_{j_0}$ equal $2$, the winding numbers around any other hole in the bounded component of $\mathbb{C} \setminus L_0$ are at most $1$, and the winding numbers around
$\mathcal{C}_{k_0}$ equal $1$. Hence, the free homotopy classes of the mentioned elements must intersect both components of $\mathbb{C} \setminus L_0$, hence they intersect $L_0$.

Representatives of any product of two different elements of  $\mathcal{E}_4'$  have winding numbers around $\mathcal{C}_{j_0}$ at least $3$, the winding numbers around any other hole in the bounded component of $\mathbb{C} \setminus L_0$ are at most $1$, and the winding numbers around
$\mathcal{C}_{k_0}$ equal $2$.  Hence, the free homotopy classes of these elements intersect $L_0$.
We obtained the following
\smallskip

\noindent {\bf Claim 12.3} {\it If $X$ is a planar domain and $\mathcal{E}$ is a standard system of generators of  $\pi_1(X,q_0)$, then there is a set $\mathcal{E}_4'\subset \pi_1(X,q_0)$, such that each of the elements of $\mathcal{E}_4'$ is the product of at most $4$ elements of $\mathcal{E}\cup \mathcal{E}^{-1}$, and the free homotopy classes of the product of at most two elements of $\mathcal{E}_4'$ intersect $L_0$.
Moreover, $e$ and $e'$ are products of at most three factors, each an element of  $\mathcal{E}_4'\cup \mathcal{E}_4'^{-1}$. }
\smallskip

\noindent The proof of the proposition in the planar case is finished as follows.
Let $\alpha_q$ be a curve in $X$ with initial point $q_0$ and terminal point $q$,
and $M$ a complex affine mapping, such that $(M\circ F)(q)\in C_3(\mathbb{R})\diagup \mathcal{S}_3$. Since $M\circ F$ is irreducible, the monodromies modulo center of $M\circ F$ along the elements of $\mbox{Is}_{\alpha}(\mathcal{E}_4')$ are not (trivial or non-trivial) powers of a single element $\sigma_j\diagup \mathcal{Z}_3$. Hence,
for each element of $\mbox{Is}_{\alpha}(\mathcal{E}_4')$ there
exists another element of $\mbox{Is}_{\alpha}(\mathcal{E}_4')$ so that
the second option of
Lemma \ref{lemfin8a} holds for this pair of elements of $\mbox{Is}_{\alpha}(\mathcal{E}_4')$. Therefore, the monodromy modulo center of $M\circ F$ along each element of $\mbox{Is}_{\alpha}(\mathcal{E}_4')$ is the product of at most two elements ${\sf b}_j\in\mathcal{B}_3\diagup \mathcal{Z}_3$ of $\mathcal{L}_-$ not exceeding $2\pi \lambda_8(X)$, and the monodromy modulo center of $M\circ F$ along each element
$\mbox{Is}_{\alpha}(\mathcal{E})$ is the product of at most $6$ elements ${\sf b}_j\in \mathcal{B}_3\diagup \mathcal{Z}_3$ with $\mathcal{L}_-(\vartheta({\sf b}_j))$ not exceeding $2\pi \lambda_8(X)$.
Proposition \ref{propfin4} is proved in the planar case.

\noindent {\bf 3. The general case.}
Since not all monodromies are powers of a single element of $\mathcal{B}_3\diagup \mathcal{Z}_3$ that is either periodic or reducible, there exists a pair of generators $e_0'$, $e_0''$ in $\mathcal{E}$, such that the monodromies along them are not powers of a single periodic or reducible element. Consider the projection $\omega^{\langle  e_0', e_0''\rangle}: \tilde{X}\to X(\langle  e_0', e_0''\rangle)$.
By the proof for tori with a hole or for $\mathbb{P}^1$ with three holes there exist a relatively closed curve $L_{\langle  e_0', e_0''\rangle}$ in $X(\langle  e_0', e_0''\rangle)$ and a M\"obius transformation $M$,
such that for $F=M\circ f$ the mapping $F_{\langle  e_0', e_0''\rangle}  =F \circ \omega_{\langle  e_0', e_0''\rangle}$ takes $L_{\langle  e_0', e_0''\rangle}$ into $\mathcal{H}$, and takes a chosen point $q_{\langle  e_0', e_0''\rangle}\in L_{\langle  e_0', e_0''\rangle}$ to a point in $C_3(\mathbb{R})\diagup \mathcal{S}_3$.

Choose a point $\tilde{q}\in \tilde X$, for which $\omega^{\langle e'_0, e''_0\rangle}(\tilde{q})=q_{\langle e'_0, e''_0\rangle}$. Let $\tilde{\alpha}$ be a curve in $\tilde X$ with initial point $\tilde{q}_0$ and terminal point $\tilde{q}$.
Then $\alpha_{\langle e_0', e''_0\rangle}\stackrel{def}=\omega^{\langle e'_0, e''_0\rangle}(\tilde{\alpha})$ is a curve in $X(\langle e'_0, e''_0\rangle)$ with initial point $(q_0)_{\langle e'_0, e''_0\rangle}$ and terminal point $ q_{\langle e'_0, e''_0\rangle}$, and the curve $\alpha_{\langle e'_0, e''_0\rangle}$ in $X(\langle e'_0, e''_0\rangle)$ and the point $\tilde{q}$ in the universal covering $\tilde X$ of $X(\langle e'_0, e''_0\rangle)$ are compatible. Put $\alpha=\omega_{\langle e'_0, e''_0\rangle}(\alpha_{\langle e'_0, e''_0\rangle})$, and for each $e_0\in \pi_1(X,q_0)$ we denote as before the element ${\rm Is}_{\alpha}({ e}_0)$ by $e$.

Put $\mathcal{E}^* \stackrel{def}= \{e_0',e_0''\}$. For an element $e_0 \in \mathcal{E}^*$
we define  $({\sf e}_0')_{\langle e'_0, e''_0\rangle}$ and $({\sf e}''_0)_{\langle e'_0, e''_0\rangle}$ as in the end of paragraph 3.1 of the proof of Proposition \ref{propfin2}.
The Riemann surface $X(\langle e'_0, e''_0\rangle)$ is a torus with a hole or $\mathbb{P}^1$ with three holes. This implies the following fact.\\
{\it There are elements ${\sf e}_0'$ and ${\sf e}_0''$, one of them contained in $\mathcal{E}^*$
or equal to the product of at most two factors among the $e'_0$ and $e''_0$, the second either contained in $\mathcal{E}^*\cup (\mathcal{E}^*)^{-1}$, or equal to the product of at most three factors among the $e'_0$ and $e''_0$,
such that the free homotopy classes of  $({\sf e}_0')_{\langle e'_0, e''_0\rangle}$, of $({\sf e}''_0)_{\langle e'_0, e''_0\rangle}$, and of their product intersect  $ L_{\langle e'_0, e''_0\rangle}$.
Moreover, $e'_0$ and $e''_0$ are products of at most three factors, each being either $({\sf e}'_0)^{\pm 1}$ or $({\sf e}''_0)^{\pm 1}$.}

\noindent Put ${\sf e}'_{\langle e'_0, e''_0\rangle}={\rm Is}_{\alpha_{\langle e'_0, e''_0\rangle} }   (({\sf e}_0')_{\langle e'_0, e''_0\rangle})$, ${\sf e}''_{\langle e'_0, e''_0\rangle}={\rm Is}_{\alpha_{\langle e'_0, e''_0\rangle} }  ( ({\sf e}_0'')_{\langle e'_0, e''_0\rangle})$.

\noindent Since the monodromies along ${\sf e}'$ and ${\sf e}''$ are not powers of a single periodic or reducible element, by Lemma \ref{lemfin8a}  each monodromy $(F_{\langle e'_0, e''_0\rangle})_*({\sf e}'_{\langle e'_0, e''_0\rangle})= F_*({\sf e}')$ and $(F_{\langle e'_0, e''_0\rangle})_*({\sf e}''_{\langle e'_0, e''_0\rangle})= F_*({\sf e}'')$ is the product of at most two elements ${\sf b}_j\in\mathcal{B}_3\diagup \mathcal{Z}_3$
with $\mathcal{L}_-(\vartheta({\sf b}_j))\leq 2\pi \lambda_5(X)$. Since $e'$ and $e''$ are products of at most three elements among $({\sf e}')^{\pm 1}$ and  $({\sf e}'')^{\pm 1}$,
each of the monodromies $F_*(e')$ and $F_*(e'')$ is the product of at most $6$  elements ${\sf b}_j\in\mathcal{B}_3\diagup \mathcal{Z}_3$
with $\mathcal{L}_-(\vartheta({\sf b}_j))\leq 2\pi \lambda_5(X)$.

Take any element $e_0\in \mathcal{E}\setminus\{e'_0,e''_0\}$. Let $e=\mbox{Is}_{\alpha}(e_0)$. Either the pair of monodromies ($F_*({\sf e}')$, $F_*(e)$) or  the pair of monodromies ($F_*({\sf e}'')$, $F_*(e)$) does not consist of two powers of the same element of $\mathcal{B}_3\diagup \mathcal{Z}_3$ that is either periodic or reducible. Suppose this is so for  the pair ($F_*({\sf e}')$, $F_*(e)$).

Let $L_{\langle {\sf e}_0'\rangle}$ be the connected component of $(\omega^{\langle {\sf e}_0', {\sf e}''_0\rangle}_{\langle {\sf e}'_0\rangle})^{-1}(L_{\langle {\sf e}'_0,{\sf e}''_0\rangle})\,$ that contains $\;\omega^{\langle {\sf e}'_0\rangle}(\tilde{q})\,$. By the analog of Lemma \ref{lemfin1}, applied to the holomorphic projection $\;\;\tilde{X}\diagup ({\rm Is}^{\tilde{q}_0})^{-1}(\langle {\sf e}'_0\rangle) \to X(\langle {\sf e}'_0, {\sf e}''_0    \rangle)\,$, the free homotopy class $\reallywidehat{({\sf e}_0')_{\langle {\sf e}_0'\rangle}}$ intersects $L_{\langle {\sf e}_0'\rangle}$. (For the definition of $({\sf e}'_0)_{\langle {\sf e}_0'\rangle}$ see the end of the paragraph 3.1 of the proof of Proposition \ref{propfin2}.)
As in the proof of Proposition \ref{propfin2} we consider the Riemann surface $X(\langle e_0, {\sf e}_0'\rangle)$ and the curve $L_{\langle e_0, {\sf e}_0'\rangle}= \omega^{\langle e_0, {\sf e}'_0\rangle}_ {\langle  {\sf e}'_0\rangle}(L_{\langle {\sf e}_0'\rangle}) $ (see paragraph 3.3. of the proof of Proposition \ref{propfin2}). As there we see that
the free homotopy class
$\reallywidehat{( {\sf e}_0')_{\langle e_0, {\sf e}'_0\rangle}}$ intersects $L_{\langle e_0, {\sf e}'_0\rangle}$.
The system $\big((e_0)_{\langle e_0, {\sf e}'_0\rangle}, ({\sf e}'_0)_{\langle e_0, {\sf e}'_0\rangle}\big)$ is associated to a standard bouquet of circles for $X(\langle e_0, {\sf e}'_0\rangle)$ (though the system $(e_0, {\sf  e}'_0)$ may not be part of a system of generators that is associated to a standard bouquet of circles  for $X$).
This can be seen in the same way
as in the proof of Proposition \ref{propfin2}.
We will apply now the arguments, used for $X(\langle e_0',e''_0\rangle )$ and the generators $(e'_0)_{\langle e'_0, e''_0\rangle},(e''_0)_{\langle e'_0, e''_0\rangle}$ of the fundamental group $\pi_1(X(\langle e'_0, e''_0\rangle), q_{\langle e'_0, e''_0\rangle})$,  to
$X(\langle e_0, {\sf e}'_0\rangle)$ and the generators  $(e_0)_{\langle e_0, {\sf e}_0'\rangle}, ({\sf e}'_0)_{\langle e_0, {\sf e}'_0\rangle}$ of the fundamental group $\pi_1(X(\langle e_0, {\sf e}'_0\rangle),   q_{\langle e_0, {\sf e}'_0\rangle})$.

In the case when $X(\langle e_0, {\sf e}'_0\rangle)$ is a torus with a hole, the intersection number of $\reallywidehat{({\sf e}'_0)_{\langle e_0, {\sf e}'_0\rangle} }$ with $L_{\langle e_0, {\sf e}'_0\rangle}$ is non-zero. Put $\mathfrak{e}'_0={\sf e}'_0$. For one of the choices $e_0^{\pm 1}$, or ${\sf e}_0'\, e_0$, denoted by $\mathfrak{e}''_0$, the free homotopy classes of $(\mathfrak{e}'_0)_{\langle e_0, {\sf e}'_0\rangle}$, $(\mathfrak{e}''_0)_{\langle e_0, {\sf e}'_0\rangle}$, and of their product intersect $L_{\langle e_0, {\sf e}'_0\rangle}$.  Since by Claim 12.1 and Claim 12.2 the element ${\sf e}'_0$ is the product of at most three elements of $\mathcal{E}$, the element  ${\sf e}''_0$ is the product of at most $1+ 3=4$ elements  of $\mathcal{E}$.
Moreover, $e_0$ is the product of at most two factors, each being $(\mathfrak{e}'_0)^{\pm 1}$, or $(\mathfrak{e}''_0)^{\pm 1}$. Hence, $e_0$ is the product of at most $7$ elements of $\mathcal{E}$.

In case $X(\langle e_0, {\sf e}_0'\rangle)$ is planar, the curve $L_{\langle e_0, {\sf e}'_0\rangle}$ must have limit points on the hole that corresponds to the generator $({\sf e}'_0)_{\langle e_0, {\sf e}'_0\rangle}$ of the fundamental group $\pi_1(X(\langle e_0, {\sf e}'_0\rangle), q_{\langle e_0, {\sf e}'_0\rangle})$. We find elements $\mathfrak{e}'_0$ and $\mathfrak{e}''_0$ such that $\mathfrak{e}'_0= {\sf e}'_0 $ and $\mathfrak{e}''_0$ is either equal to $e_0^{-1}$,
or to the product of at most three factors, one being equal to $e_0$ and the others equal to ${\sf e}' _0$, and the free homotopy classes of $(\mathfrak{e}'_0)_{\langle e_0, {\sf e}'_0\rangle}$, $(\mathfrak{e}''_0)_{\langle e_0, {\sf e}'_0\rangle}$, and their product intersect  $L_{\langle e_0, {\sf e}'_0\rangle}$. Moreover,
$e_0$ is the product of at most $3$ factors, each being equal to $(\mathfrak{e}''_0)^{\pm 1}$ or
$(\mathfrak{e}'_0)^{\pm 1}$.
Since ${\sf e}_0'$ is the product of at most three elements of $\mathcal{E}$, ${\mathfrak e}''_0$ is the product of at most $1+2\cdot 3=7$  elements of $\mathcal{E}$.
In both cases for $X(\langle e_0, {\sf e}'_0\rangle)$ the elements $\mathfrak{e}'_0$, $\mathfrak{e}''_0$, and $\mathfrak{e}' _0\mathfrak{e}''_0$ are products of at most $10$ elements of $\mathcal{E} \cup \mathcal{E}^{-1}$.

Since for $e={\rm Is}_{\alpha}(e_0)$, and ${\sf e}'={\rm Is}_{\alpha}({\sf e}'_0)$
 the monodromies $F_*(e)$ and $F_*({\sf e}'_0)$ are not powers of a single periodic or reducible element,  for $\mathfrak{e}'={\rm Is}_{\alpha}(\mathfrak{e}' _0)$ and  $\mathfrak{e}''={\rm Is}_{\alpha}(\mathfrak{e}''_0)$ the monodromies
$F_*(\mathfrak{e}')$ and $F_*(\mathfrak{e}'')$  cannot be powers of a single element  that is either periodic or reducible.
Lemma \ref{lemfin8a} implies, that $F_*(\mathfrak{e}')$ and $F_*(\mathfrak{e}'')$ are products of at most two factors $\sf b$ with $\mathcal{L}_-(\vartheta({\sf b}))$ not exceeding $2\pi \lambda_{10}(X)$. Hence, $F_*(e)$ is the product of at most $6$ factors $\sf b$ with $\mathcal{L}_-(\vartheta({\sf b}))$ not exceeding $2\pi \lambda_{10}(X)$.
We obtain the statement of Proposition \ref{propfin4} in the general case.
Proposition \ref{propfin4} is proved. \hfill $\Box$

\medskip

\noindent {\bf Proof of Theorem \ref{thmfin3}.}
Let $X$ be a connected Riemann surface of genus $\sf g$ with ${\sf m}+1\geq 1$ holes. Since each holomorphic $(0,3)$-bundle with a holomorphic section on $X$ is isotopic to a holomorphic special $(0,4)$-bundle, we need to estimate the number of isotopy classes of irreducible smooth special $(0,4)$-bundles on $X$, that contain a holomorphic bundle.
By Lemma \ref{lemEl.0}
the monodromies of an irreducible bundle are not powers of a single element of $\mathcal{B}_3 \diagup \mathcal{Z}_3$ which is conjugate to a $\sigma_j\diagup \mathcal{Z}_3$, but they may be powers of a single periodic element of $\mathcal{B}_3 \diagup \mathcal{Z}_3$ (equivalently, the isotopy class may contain a locally holomorphically trivial holomorphic bundle).

Consider an irreducible special holomorphic $(0,4)$-bundle on $X$
which is not isotopic
to a locally holomorphically trivial bundle.
Let $F(x),\, x \in X,$ be the set of finite distinguished points in the fiber over $x$.
By the Holomorphic Transversality Theorem \cite{KZ}
the mapping $F:X\to C_3(\mathbb{C})\diagup \mathcal{S}_3$ can be approximated on relatively compact subsets of $X$ by holomorphic mappings
that are transverse to $\mathcal{H}$.  Similarly as in the proof of Theorem \ref{thmfin1}
we will therefore assume in the following (after slightly shrinking $X$ to a deformation retract of $X$ and approximating $F$) that $F$ is transverse to $\mathcal{H}$.

By Proposition \ref{propfin4}
there exists a complex affine mapping $M$ and a point $q\in X$ such that $M\circ
F(q)$ is contained in $C_3(\mathbb{R})\diagup \mathcal{S}_3$, and for an arc $\alpha$ in $X$ with initial point $q_0$ and terminal point $q$ and each element $e_j\in {\rm Is}_{\alpha}(\mathcal{E})$ the monodromy $(M\circ F)_*(e_j)\diagup \mathcal{Z}_3$ of the bundle can be written as product of at most $6$ elements  $\textsf{b}_{j,k},\, k=1,2,3,4,5,6,$ of $\mathcal{B}_3 \diagup \mathcal{Z}_3$ with
\begin{equation}\label{eqfin7b}
\mathcal{L}_-(\vartheta(\textsf{b}_{j,k})) \leq 2\pi \lambda_{10}(X).
\end{equation}
\noindent The mappings $F$ and $M\circ F$ from $X$ into the symmetrized configuration space are free homotopic.

Consider an isotopy class of special $(0,4)$-bundles that corresponds to a conjugacy class of homomorphisms $\pi_1(X,q_0)\to \mathcal{B}_3\diagup \mathcal{Z}_3$ whose image is generated by a single periodic element of $\mathcal{B}_3\diagup \mathcal{Z}_3$.
Up to conjugacy we may assume that this element is one of the following:
$\Delta_3\diagup \mathcal{Z}_3,\, (\sigma_1 \sigma_2 )\diagup \mathcal{Z}_3 , \,(\sigma_1 \sigma_2)^{-1} \diagup \mathcal{Z}_3$, or ${\rm Id}$ . For each of these elements $\textsf{b}$ the equality $\mathcal{L}_-(\vartheta(\textsf{b}))=0$ holds. Hence, in this case the isotopy class contains a smooth mapping $\tilde F$ such that for each $e_{j,0}\in \mathcal{E}$ the monodromy $(M\circ F)_*(e_{j,0})\diagup \mathcal{Z}_3$ of the bundle can be written as product of at most $6$ elements  $\textsf{b}_{j,k},\, k=1,2,3,4,5,6,$ of $\mathcal{B}_3 \diagup \mathcal{Z}_3$ satisfying inequality \eqref{eqfin7b}.

The same argument as in the proof of Theorem \ref{thmfin1} shows the following fact. Each irreducible free homotopy class of mappings $X\to C_3(\mathbb{C})\diagup \mathcal{S}_3$ that contains a holomorphic mapping contains a smooth mapping $\tilde F$ such that for each $e_{j,0}\in \mathcal{E}$ the monodromy $\tilde{F}_*(e_{j,0})\diagup \mathcal{Z}_3$ of the bundle can be written as product of at most $6$ elements  $\textsf{b}_{j,k},\, k=1,2,3,4,5,6,$ of $\mathcal{B}_3 \diagup \mathcal{Z}_3$ satisfying inequality \eqref{eqfin7b}.

Using Lemma  \ref{lem10.1a}  (see also Lemma 1 of \cite{Jo3}) the number of elements $  \textsf{b}\in \mathcal{B}_3 \diagup \mathcal{Z}_3$ (including the identity),
for which $\mathcal{L}_-(\vartheta(\textsf{b}))\leq 2\pi \lambda_{10}(X)$,
is estimated as follows.
The element ${\sf w}\stackrel{def}=\vartheta({\sf b}) \in \mathcal{PB}_3 \diagup \mathcal{Z}_3$ can be considered as a reduced word in the free group generated by $a_1=\sigma_1^2\diagup \mathcal{Z}_3$ and $a_2=\sigma_2^2\diagup \mathcal{Z}_3$.
By Lemma \ref{lem10.1a}
there are no more than $ \frac{1}{2}\exp(6 \pi \lambda_{10}(X))+1\leq \frac{3}{2}\exp(6 \pi \lambda_{10}(X))$ reduced words $\sf w$ in $a_1$ and $a_2$
(including the identity) satisfying the inequality $\mathcal{L}_-(\textsf{w})\leq 2\pi \lambda_{10}(X)$.

For a given element ${\sf w}\in \mathcal{PB}_3 \diagup \mathcal{Z}_3$ (including the identity) we describe now all elements ${\textsf{b}}$
of $\mathcal{B}_3 \diagup \mathcal{Z}_3$ with
$\vartheta(\textsf{b})=\textsf{w}$.
If ${\textsf{w}}\neq \mbox{Id}$ these are
the following elements.
If the first term of ${\sf w}$ equals
$a_j^k$ with $k\neq 0$, then the possibilities are $\textsf{b}={\sf w} \cdot (\Delta_3^{\ell}\diagup \mathcal{Z}_3)$ with $\ell=0$ or $1$,  $\textsf{b}=(\sigma_j^{{\rm{sgn}} k}\diagup\mathcal{Z}_3)\cdot  {\sf w}  \cdot (\Delta_3^{\ell}\diagup \mathcal{Z}_3)$ with $\ell=0$ or $1$, or  $\textsf{b}=(\sigma_{j'}^{\pm 1}\diagup \mathcal{Z}_3) \cdot {\sf w} \cdot (\Delta_3^{\ell}\diagup \mathcal{Z}_3)$ with $\ell=0$ or $1$ and $\sigma_{j'}\neq
\sigma_{j}$. Hence, for $\textsf{w}\neq \mbox{Id}$ there are $8$ possible choices of elements ${\textsf{b}}\in \mathcal{B}_3 \diagup \mathcal{Z}_3$ with $\vartheta({\textsf{b}})= {\sf w} $.

If $\textsf{b}= \mbox{Id}$ then the choices are $\Delta^{\ell}\diagup \mathcal{Z}_3$ and $(\sigma_j^{\pm 1}\Delta^{\ell})\diagup \mathcal{Z}_3$ for $j=1,2,$ and $\ell=0$ or $\ell=1$. These are $10$ choices.  Hence, there are no more than  $15 \exp( 6 \pi \lambda_{10}(X))$ different
elements $\textsf{b}\in \mathcal{B}_3 \diagup \mathcal{Z}_3$  with $\mathcal{L}_-(\vartheta(\textsf{b}))\leq 2 \pi \lambda_{10}(X)$.

Each monodromy is the product of at most six elements ${\sf{b}}_j$ of
$\mathcal{B}_3 \diagup \mathcal{Z}_3$ with $\mathcal{L}_-(\vartheta(\textsf{b}_j))\leq 2 \pi \lambda_{10}(X)$.
Hence, for each monodromy there are no more than $(15 \exp( 6 \pi \lambda_{10}(X)))^{6}$
possible choices. We proved that there are up to isotopy no more than $(15 \exp( 6 \pi \lambda_{10}(X)))^{6(2g+m)}$
irreducible holomorphic $(0,3)$-bundles with a holomorphic section over $X$.
Theorem \ref{thmfin3} is proved. \hfill $\Box$
\smallskip

Notice that we proved a slightly stronger statement, namely, over a Riemann surface of genus $g$ with $m+1\geq 1$ holes there are no more than
$(15 \exp( 6 \pi \lambda_{10}(X)))^{6(2g+m)}$ isotopy classes of smooth $(0,3)$-bundles with a
smooth section that contain a holomorphic bundle with a holomorphic section that is either irreducible or isotopic to the trivial bundle.

\medskip

\noindent {\bf Proof of Theorem \ref{thmfin2}.} Proposition \ref{propEl.2} and Theorem \ref{thmfin3} imply Theorem \ref{thmfin2} as follows. Suppose an isotopy class of smooth $(1,1)$-bundles over a finite open Riemann surface $X$ contains a holomorphic bundle. By Proposition \ref{propEl.2} the class contains a holomorphic bundle which is the double branched covering of a holomorphic special $(0,4)$-bundle.
If the $(1,1)$-bundle is irreducible then also the $(0,4)$-bundle is irreducible.
There are up to isotopy no more than $\big(15(\exp(6 \pi \lambda_{10}(X)))\big)^{6(2g+m)}$
holomorphic special $(0,4)$-bundles over $X$ that are either irreducible or isotopic to the trivial bundle.

By Theorems \ref{thmEl1} and \ref{thmfin3} there are no more than $\big(15\big(\exp(6 \pi \lambda_{10}(X)))\big)^{6(2g+m)}$
conjugacy classes of monodromy homomorphisms that correspond to
a special holomorphic $(0,4)$-bundle over $X$
that is either irreducible or isotopic to the trivial bundle.
Each monodromy homomorphism of the holomorphic double branched covering is a lift of the respective monodromy homomorphism of the holomorphic special $(0,4)$-bundle.
Different lifts of a monodromy mapping class of a special $(0,4)$-bundle differ by involution, and the fundamental group of $X$ has $2g+m$ generators.
Using Theorem \ref{thmEl1} for $(1,1)$-bundles, we see that
there are no more than  $2^{2g+m}\big(15(\exp(6 \pi \lambda_{10}(X)))\big)^{6(2g+m)}=\big(2 \cdot 15^6\cdot\exp(36 \pi \lambda_{10}(X))\big)^{2g+m}$
isotopy classes of $(1,1)$-bundles that contain a holomorphic bundle that is either irreducible or isotopic to the trivial bundle.
Theorem \ref{thmfin2} is proved.
\hfill $\Box$

\medskip

\noindent For convenience of the reader we give the short reduction of the Corollaries \ref{corfin1a} and \ref{corfin1b} to Theorems \ref{thmfin1} and \ref{thmfin2}, respectively.
The arguments of the reduction
are known in principle, but the case  considered here is especially simple.

\medskip

\noindent {\bf Proof of Corollary \ref{corfin1a}.}
We will prove that on a punctured Riemann surface there are no non-constant reducible holomorphic mappings to the twice punctured complex plane and that any homotopy class of mappings from a punctured Riemann surface to the twice  punctured complex plane contains at most one holomorphic mapping. This implies the corollary.

Recall that a holomorphic mapping $f$ from any punctured Riemann surface $X$ to the twice punctured complex plane extends by Picard's Theorem to a meromorphic function $f^c$ on the closed  Riemann surface $X^c$. \index{Picard ! Theorem of}
Suppose now that $X$ is a punctured Riemann surface and that the mapping $f:X \to \mathbb{C}\setminus \{-1,1\}$ is reducible, i.e. it is homotopic to a mapping into a punctured disc contained in $\mathbb{C}\setminus \{-1,1\}$. Perhaps after composing $f$ with a M\"obius transformation we may suppose that this puncture equals $-1$.  Then the meromorphic extension $f^c$ omits the value $1$. Indeed, if $f^c$ was equal to $1$ at some puncture of $X$, then $f$ would map
the boundary of a sufficiently small disc on $X^c$
that contains the puncture to a loop in $\mathbb{C}\setminus \{-1,1\}$ with non-zero winding number around $1$ ,
which contradicts the fact that $f$ is homotopic to a mapping into a disc punctured at $-1$ and contained in $\mathbb{C}\setminus \{-1,1\}$. Hence, $f^c$ is a meromorphic function on a compact Riemann surface that omits a value, and, hence $f$ is constant. Hence, on a punctured Riemann surface there are no non-constant reducible holomorphic mappings to $\mathbb{C}\setminus \{-1,1\}$.

Suppose $f_1$ and $f_2$ are non-constant homotopic holomorphic mappings from the punctured Riemann surface $X$ to the twice punctured complex plane. Then for their meromorphic extensions $f^c_1$ and $f^c_2$ the functions  $f^c_1-1$ and $f^c_2-1$ have the same divisor on the closed Riemann surface $X^c$. Indeed, suppose, for instance, that $f^c_1-1$ has a zero of order $k >0$
at a puncture $p$. Then for the boundary $\gamma$ of a small disc in $X^c$
around $p$
the curve $(f_1-1) \circ \gamma$ in $\mathbb{C}\setminus \{-2,0\}$ has index $k$
with respect to the origin.
Since $f_2-1$ is homotopic to $f_1-1$ as mapping to $\mathbb{C}\setminus \{-2,0\}$, the curve  $(f_2-1) \circ \gamma$ is free homotopic to $(f_1-1) \circ \gamma$ . Hence, $f_2-1$ has a zero of order $k$
at $p$. Applying the same arguments with $0$ replaced by $\infty$, we obtain that $f^c_1-1$ and $f^c_2-1$ have the same divisor. Hence, $f^c_1-1$ and $f^c_2-1$ differ by a non-zero multiplicative constant. 
Since the functions are non-constant they must take the value $-2$. By the same reasoning as above the functions are equal to $-2$ simultaneously. Hence, the multiplicative constant is equal to $1$.
We proved that non-constant homotopic holomorphic maps from punctured Riemann surfaces
to $\mathbb{C}\setminus \{-1,1\}$ are equal. \hfill $\Box$

\medskip
\noindent {\bf Proof of Corollary \ref{corfin1b}.}
We need the following fact. For each special $(0,4)$-bundle $\mathfrak{F}=(X\times \mathbb{P}^1, {\rm pr}_1, {\mathbold E}, X)$ there is a finite unramified covering $\hat{{\sf P}}:\hat X\to X$ of $X$, such that $\mathfrak{F}$ lifts to a  special $(0,4)$-bundle
$(\hat{X}\times\mathbb{P}^1,{\rm pr}_1,\hat{\mathbold{E}},X)$,
for which the complex curve  $\hat{\mathbold{E}}$ is the union of four disjoint complex curves $\hat{\mathbold{E}}^k,\, k=1,2,3,4,$ each intersecting each fiber $\{\hat{x}\}\times\mathbb{P}^1$ along a single point $(\hat{x},\,\hat{g}^k(\hat{x}))$. 
This can be seen as follows. Let $q_0$ be the base point of $X$.
The monodromy mapping class along each element $e$ of $\pi_1(X,q_0)$ takes the set of distinguished points $\hat{\mathbold{E}}\cup (\{q_0\}\times\mathbb{P}^1)$
onto itself, permuting them by a permutation $\sigma(e)$. Consider the set $ N$ of
elements $e\in \pi_1(X,q_0)$ for which $\sigma(e)$ is the identity. The set $N$ is a normal subgroup of $\pi_1(X,q_0)$. Its index is finite, since two left cosets $e_1\,N$
and $e_2 \, N$ are equal if $ \sigma(e_2\,e_1^{-1}) =    \sigma(e_2)\sigma(e_1)^{-1}={\rm Id}$, and there are only finitely many distinct permutations of points of $\hat{\mathbold{E}}\cup (\{q_0\}\times\mathbb{P}^1)$. The quotient $\hat{X}\stackrel{def}=\tilde {X}\diagup {\rm Is}^{\tilde{q}_0}(N)$ of the universal covering of $X$ by the group of covering transformations corresponding to $N$ and the canonical projection $\hat{X}\to X$ define the required covering.

To prove the corollary, we have to show first,  that any reducible holomorphic $(1,1)$-bundle over a punctured Riemann surface $X$ is locally holomorphically trivial and secondly, that two isotopic (equivalently, smoothly isomorphic) holomorphically non-trivial holomorphic $(1,1)$-bundles over $X$ are holomorphically isomorphic.

The second fact is obtained as follows.
Suppose the holomorphically non-trivial holomorphic $(1,1)$-bundles $\mathfrak{F}_j,\,j=1,2,$ have conjugate monodromy homomorphisms. By Proposition \ref{propEl.2}
each bundle $\mathfrak{F}_j,\, j=1,2,$ is holomorphically isomorphic to a double branched covering of a special holomorphic $(0,4)$-bundle
$(X\times\mathbb{P}^1,{\rm pr}_1,\mathbold{E}_j,X)\stackrel{def}= {\sf Pr}(\mathfrak{F}_j)$. The bundles ${\sf Pr}(\mathfrak{F}_j)$ are isotopic, since they have conjugate monodromy homomorphisms. There is a finite unramified covering $\hat{{\sf P}}:\hat X\to X$ of $X$, such that the bundles  ${\sf Pr}(\mathfrak{F}_j)$ have isotopic lifts
$(\hat{X}\times\mathbb{P}^1,{\rm pr}_1,\hat{\mathbold{E}}_j,X)$ to $\hat X$, and for each $j$ the complex curve  $\hat{\mathbold{E}}_j$ is the union of four disjoint complex curves $\hat{\mathbold{E}}_l^k,\, k=1,2,3,4,$ each intersecting each fiber $\{\hat{x}\}\times\mathbb{P}^1$ along a single point $(\hat{x},\,\hat{g}_j^k(\hat{x}))$. We may choose the label of the points so that the isotopy moves  the complex curve $\hat{\mathbold{E}}_1^k$ to $\hat{\mathbold{E}}_2^k$, $k=1,2,3,4$.
The lifted bundles are not isotopic to the trivial bundle. The mappings $\hat{X}\ni \hat{x}\to \hat{g}_j^k(\hat{x})\in \mathbb{P}^1$ are holomorphic.
We may assume that $\hat{g}_j^4(\hat{x})=\infty$ for each $\hat{x}$.
Define for $j=1,2,$ a holomorphic isomorphism of the bundle $(\hat{X}\times\mathbb{P}^1,{\rm pr}_1,\hat{\mathbold{E}}_j,X)$ by
$$
\{\hat{x}\}\times \mathbb{P}^1\ni (\hat{x},\,\zeta)\to \Big(\hat{x},\, \frac{\hat{g}_j^1(\hat{x})-\zeta}{\hat{g}_j^1(\hat{x})-\hat{g}_j^2(\hat{x}) }\Big)\,.
$$

The image $\hat{\mathbold{E}}'_j$ of $\hat{\mathbold{E}}_j$ under the $j$-th isomorphism
intersects the fiber over each $\hat{x}\in \hat X$ along the four points $(\hat{x},0),\,(\hat{x},1),\,(\hat{x},\infty),$ and $(\hat{x},\mathring{g}_j(\hat{x}))$ for a holomorphic mapping $\mathring{g}_j:\hat X \to \mathbb{P}^1$
that avoids $0,\,1$ and $\infty$. The mappings $\;\mathring{g}_j\;$, $j=1,2,\;$ are homotopic, since the bundles
are isotopic. They are not homotopic to a constant mapping since the bundles are not isotopic to the trivial bundle.
By Corollary \ref{corfin1a} the mappings $\;\mathring{g}_1\;$ and \;$\mathring{g}_2\;$ coincide.
Hence, the bundles $\;\;(\hat{X}\times\mathbb{P}^1,{\rm pr}_1,\hat{\mathbold{E}}_j,X)\;\;$
are holomorphically isomorphic to the bundle $\mathfrak{F}'= (\hat{X}\times\mathbb{P}^1,{\rm pr}_1,\hat{\mathbold{E}},X)\;\;$ with
$\;\hat{\mathbold{E}} \cap \big(\{\hat{x}\}\times\mathbb{P}^1\big)=\{\hat{x}\}\times
(0,1,\mathring{g}(\hat{x}),\infty)\;$ for a holomorphic mapping $\mathring{g}(\hat{x}):\hat{X}\to \mathbb{C}\setminus\{0,1\}$.

For each bundle ${\sf Pr}(\mathfrak{F}_j), \, j=1,2,$ a holomorphically isomorphic bundle is obtained from $\mathfrak{F}'$ by gluing for each $x\in X$ the fibers over $\hat{x}\in (\hat{\sf P})^{-1}(x)$
together using a complex affine mapping determined by ${\sf Pr}(\mathfrak{F}_j)$. The complex affine mappings depend holomorphically on $\hat x$.  Each of them takes the unordered triple of finite distinguished points in one fiber to the unordered triple of finite distinguished points in the other fiber over $\hat x$. Assign to each finite distinguished point
the number $k$ of the curve $\hat{\mathbold{E}}^k$ in $\hat{\mathbold{E}}$ to which the point belongs.
Each of the complex affine mappings
is uniquely determined by the permutation of these numbers $k$, induced by it. Since the bundles  ${\sf Pr}(\mathfrak{F}_j)$ are smoothly isomorphic, the respective permutations for the two bundles coincide. Hence, the bundles
${\sf Pr}(\mathfrak{F}_j)$ are holomorphically isomorphic to a single bundle that is obtained from  $\mathfrak{F}'$ by a holomorphic gluing procedure of the fibers over points $\hat{x}\in (\hat{\sf P})^{-1}(x)$
 for $x\in X$.

For each bundle ${\sf Pr}(\mathfrak{F}_j)$  there exist finitely many double branched coverings, and if the given double branched coverings of  ${\sf Pr}(\mathfrak{F}_1)$ and ${\sf Pr}(\mathfrak{F}_2)$, respecticvely, have conjugate monodromy homomorphisms, they are holomorphically isomorphic.
Since the bundles $\mathfrak{F}_j$, $j=1,2,$ are double branched coverings of the ${\sf Pr}(\mathfrak{F}_j)$ and have conjugate monodromy homomorphism, they are holomorphically isomorphic. The second fact is proved.

The first fact is obtained as follows. After a holomorphic isomorphism we may assume
that the reducible holomorphic $(1,1)$-bundle is a double branched covering of a reducible special $(0,4)$-bundle ${\sf Pr}(\mathfrak{F})=(X\times \mathbb{P}^1, {\rm pr}_1, \mathring{\mathbold{E}}\cup \mathbold{s}^{\infty}, X)$. After a further isomorphism the bundle ${\sf Pr}(\mathfrak{F})$ lifts to a holomorphic bundle $\reallywidehat{{\sf Pr}(\mathfrak{F})}= (\hat{X} \times \mathbb{P}^1, {\rm pr}_1, \hat{\mathring{\mathbold{E}}}\cup \widehat{\mathbold{s}^{\infty}},\hat{ X}) $, such that
$\hat{\mathring{\mathbold{E}}}$ intersects each fiber $\{x\}\times\mathbb{P}^1$ along a set of the form $\{\hat{x}\}\times \{0,1,\mathring{g}(\hat{x})\}$. Since the bundle $\mathfrak{F}$ is reducible, the bundle ${\sf Pr}(\mathfrak{F})$ and also $\reallywidehat{{\sf Pr}(\mathfrak{F})}$ are reducible. Hence,  the mapping $\mathring{g}$ is constant by Corollary \ref{corfin1a}. Hence, all fibers of $\reallywidehat{{\sf Pr}(\mathfrak{F})}$ are conformally equivalent, and, hence, all fibers of ${\sf Pr}(\mathfrak{F})$ are conformally equivalent. Since $\mathfrak{F}$ is the double branched covering of ${\sf Pr}(\mathfrak{F})$, all fibers of $\mathfrak{F}$ are conformally equivalent.
The first fact is proved.
\hfill $\Box$

\medskip

\noindent {\bf Proof of Proposition \ref{propfin1a}.}
Denote by $S^{\alpha}$ a skeleton of $T^{\alpha,\sigma} \subset T^{\alpha}$ which is the union of two circles each of which lifts
under the covering ${\sf P}:\mathbb{C}\to T^{\alpha}$ to
a straight line segment which is parallel to an axis in the complex plane. Denote the intersection point of the two circles by $q_0$.
Note that $S^{\alpha}$ is a standard bouquet of circles for $T^{\alpha,\sigma}$ with base point $q_0$, and
${\sf P}^{-1}(T^{\alpha,\sigma})$ is the $\frac{\sigma}{2}$-neighbourhood of
${\sf P}^{-1}(S^{\alpha})$. We may assume that ${\sf P}^{-1}(q_0)$ is the lattice $\mathbb{Z}+i\alpha \mathbb{Z}$.

Denote by $e$ the generator of $\pi_1(T^{\alpha,\sigma},q_0)$, that lifts to a vertical line segment and $e'$  the generator of $\pi_1(T^{\alpha,\sigma},q_0)$, that lifts to a horizontal line segment. Put $\mathcal{E}=\{e,e'\}$. We show first the inequality
\begin{equation}\label{eqfin20}
\lambda_3(T^{\alpha,\sigma}) \leq \frac{4(2\alpha+1)}{\sigma}\;.
\end{equation}
For this purpose we take any primitive element $e''$ of the fundamental group $\pi_1(T^{\alpha,\sigma},q_0)$ which is the product of at most three factors, each of the factors being an element of $\mathcal{E}$ or the inverse of an element of $\mathcal{E}$.
We represent
the element $e''$ by a piecewise $C^1$ mapping $f_1$ from an interval $[0,l_1]$
to the skeleton $S^{\alpha}$.
We may consider $f_1$ as a piecewise $C^1$ mapping from the circle
$\mathbb{R}\diagup (x \sim x+l_1)$ to the skeleton, and assume that for all points $t'$ of the circle where $f_1$ is not smooth, $f_1(t')=q_0$.
Let $t_0\in [0,l_1]$ be a point for which $f_1(t_0)\neq q_0$.
Let $\tilde{f}_1$ be a piecewise smooth mapping from $[t_0,t_0+l_1]$ to the universal covering $\mathbb{C}$ of $T^{\alpha}\subset T^{\alpha,\sigma}$
which projects to $f_1$.
We may take $f_1$ so that the equality $|\tilde{f}_1'|=1$ holds. The mapping may be chosen so that $l_1\leq 2 \alpha  + 1$.
(Recall that $\alpha\geq 1$ and the element $e$ is primitive.)

Take any $t'$ for which $f_1$ is not smooth. We may assume that $f_1$ is chosen so that
the direction of $\tilde{f}_1'$ changes by the angle $\pm \frac{\pi}{2}$ at each such point. Hence,
there exists a neighbourhood $I(t')$ of $t'$ on $(t_0,t_0+\ell_1)$, such that the restriction $\tilde{f}_1'| I(t')$
covers two sides of a square of side length $\frac{\sigma}{2}$. Denote $\tilde{q}'_0$  the common vertex  $\tilde{f}_1'(t')$ of these sides, and
by  $\tilde{q}''_0$ the vertex of the square that is not a vertex of one of the two sides.
Replace the union of the two sides of the square that contain $\tilde{q}'_0$ by a quarter-circle of radius $\frac{\sigma}{2}$ with center at the vertex $\tilde{q}''_0$, and
parameterize the latter by $t \to \frac{\sigma}{2} e^{\pm i\frac{2}{\sigma}t}$ so that the absolute value of the derivative equals $1$.
Notice that the quarter-circle is shorter than the union of the two sides.

Proceed in this way with all such points $t'$. After a reparameterization
we obtain a $C^1$ mapping $\tilde f$ of the interval $[0,l]$ of length $l$ not exceeding $2 \alpha +1$ whose image is contained in the union of
${\sf P}^{-1}(S^{\alpha})$
with some quarter-circles, such that $|\tilde{f}'|=1$.
The distance of each point of the image of $\tilde f$ to the boundary of ${\sf P}^{-1}(T^{\alpha,\sigma})$ 
is not smaller than $\frac{\sigma}{2}$. The mapping $\tilde f$ is piecewise of class $C^2$.
The normalization condition  $|\tilde{ f}'|=1$ implies $|\tilde{f}''|\leq \frac{2}{\sigma}$.

The projection $f={\sf P}\circ\tilde{f} $ can be considered as a mapping from the circle $\mathbb{R}\diagup (x\sim x +l)$ of length $l$ not exceeding $2 \alpha + 1$  to $T^{\alpha,\sigma}$,  that represents
the free homotopy class $\reallywidehat{e''}$ of the chosen element of the fundamental group.

Consider the mapping $ x+iy \to \tilde{F}(x+iy) \stackrel{def}=\tilde{f}(x) + i \tilde{f}'(x) y \in \mathbb{C}$, where
$x+iy$ runs along the rectangle $R_l= \{x+iy  \in \mathbb{C}:
x\in [0,l], |y| \leq \frac{\sigma}{4}\}$.
The image of this mapping is contained in the closure of ${\sf P}^{-1}(T^{\alpha,\sigma})$.
Since $2\frac{\partial}{\partial z} \tilde{F} (x+iy)= 2 \tilde{f}'(x) + i \tilde{f}''(x) y $ and $2\frac{\partial}{ \partial \bar z} \tilde{F }(x+iy)=  i \tilde{f}''(x) y $,
the Beltrami coefficient $\mu
_{\tilde{F}}(x+iy)= \frac{\frac{\partial}{\bar z} \tilde{F} (x+iy)}{\frac{\partial}{z} \tilde{F} (x+iy)}$ of $\tilde{F}$  satisfies the inequality $|\mu
_{\tilde{F}}(x+iy) | \leq \frac{1}{3}$.  Hence, for $K=\frac{1+\frac{1}{3}}{1-\frac{1}{3}}=2$ the mapping $\tilde{F}$ descends to a $K$-quasiconformal mapping $F$ from the annulus $A_l$ to $T^{\alpha,\sigma}$ of extremal length $\lambda(A_l)=\frac{l}{\frac{\sigma}{2}}\leq 2 \frac{(2 \alpha +1)}{\sigma}$
that represents the free homotopy class of the element $e''$ of the fundamental group $\pi_1(T^{\alpha,\sigma} ,q_0)$. Realize $A_l$ as an annulus in the complex plane. Let $\varphi$ be the solution of the Beltrami equation on $\mathbb{C}$ with Beltrami coefficient $\mu_{\tilde{F}}$ on $A_l$ and zero else.  Then the mapping $g=F \circ \varphi^{-1}$ is a holomorphic mapping of the annulus $\varphi(A_l)$ of extremal length not exceeding $K \lambda(A_l) \leq  \frac{4(2\alpha +1 )}{\sigma} $ into $T^{\alpha,\sigma}$ that represents the chosen element of the fundamental group $\pi_1(T^{\alpha,\sigma},q_0)$. Inequality \eqref{eqfin20} is proved.

By Theorem \ref{thmfin1}
for tori with a hole there are up to homotopy no more than $3(\frac{3}{2}e^{36 \pi \lambda_3(T^{\alpha,\sigma})})^2\leq \frac{27}{4}e^{2^5\cdot 3^2 \pi \frac{2\alpha +1}{\sigma}}< 7e^{2^5 \cdot 3^2 \pi \frac{2\alpha +1}{\sigma}} $
non-constant irreducible holomorphic mappings from $T^{\alpha,\sigma}$ to the twice punctured complex plane.

\smallskip

We give now the proof of the lower bound. Fix $\alpha\geq 1$ and let $\delta\leq\frac{1}{10}$.
We consider the annulus $A^{\alpha,5\delta}=\{z\in \mathbb{C}:|\mbox{Re}z|< \frac{5\delta}{2}\}\diagup (z\sim z+\alpha i)$ of extremal length equal to $\frac{\alpha}{5\delta}\geq 2\alpha$.

For any natural number $j$ we consider all elements of $\pi_1(\mathbb{C}\setminus \{-1,1\},0)$ of the form
\begin{equation}\label{eqfin21}
a_1^{\pm 2}a_2^{\pm 2} \ldots a_1^{\pm 2}a_2^{\pm 2}
\end{equation}
containing $2j$ terms, each of the form $a_j^{\pm 2}$. The choice of the sign in the exponent of each term is arbitrary. There are $2^{2j}$ elements of this kind.

Put $j=[\frac{\alpha}{60\pi \delta}]$, where $[x]$ is the largest integer not exceeding a positive number $x$. Then $\frac{\alpha}{5\delta}\geq 12 j \pi$ and
Remark \ref{rem3-braids.1} implies that
for each word $w$ of the form $a_1^{\pm 2} a_2^{\pm 2} \ldots a_1^{\pm 2} a_2^{\pm 2}$ with $N=2j$ terms there exists
a holomorphic mapping $\mathfrak{g}_w: \big(A^{\alpha,5\delta},0\diagup (z\sim z+i\alpha)\big) \to \big(\mathbb{C}\setminus \{-1,1\},0\big)$ from $A^{\alpha,5\delta}$ with base point $0\diagup (z\sim z+i\alpha)$ to  $\mathbb{C}\setminus \{-1,1\}$ with base point $0$, that represents $w$. Moreover, the image of $\mathfrak{g}_w$ is contained in the domain  $\{z\in\mathbb{C}: |z|<C, |z\pm 1|>\frac{1}{C}\}$ for a constant $C>1$.

The plan is the following. We will embed $A^{\alpha,\delta}$ into $T^{\alpha,\delta}$, restrict  $\mathfrak{g}_w$ to a smaller annulus, and
extend the pushforward of the restriction
to a smooth mapping $\mathring{\mathfrak{g}}_w$ from
$T^{\alpha,\delta}$ into $\{z\in\mathbb{C}: |z|<C, |z\pm 1|>\frac{1}{C}\}$ such that (with $\sf P$ being the projection ${\sf P}:\mathbb{C}\to T^{\alpha}$) the monodromy
along the circle ${\sf P}(\{\mbox{Re}z=0\})$
with base point ${\sf P}(0)$
is equal to \eqref{eqfin21}, and the monodromy along  ${\sf P}(\{\mbox{Im}z=0\}$
with the same base point equals the identity.
We will show the existence of a positive number $\varepsilon$ depending only on $C$, such that correcting the smooth mapping  $\mathring{\mathfrak{g}}_w$ on $T^{\alpha,\sigma}\stackrel{def}=T^{\alpha,\varepsilon\delta}$ by a solution of a $\bar{\partial}$-problem on $T^{\alpha,\sigma}$   with $L^{\infty}$-norm smaller than  $C$, we obtain a holomorphic mapping $T^{\alpha,\sigma} \to \mathbb{C}\setminus\{-1,1\}$ with
monodromies $w$ and $\rm Id$. In this way  for  $\sigma=\varepsilon \delta$ we will obtain $2^{2j}\geq \frac{1}{4}
e^{\frac{2\varepsilon\log2}{60\pi}       \frac{\alpha}{\sigma}}$ non-homotopic holomorphic mappings from $T^{\alpha,\sigma}$ to $\mathbb{C}\setminus\{-1,1\}$.
We will now give the detailed proof following the plan.

Let $\tilde{\mathfrak{g}}_w$ be the lift of $\mathfrak{g}_w$ to a mapping from the strip $\{z\in \mathbb{C}:|\mbox{Re}z|< \frac{5\delta}{2}\}$ to $\{z\in\mathbb{C}: |z|<C, |z\pm 1|>\frac{1}{C}\}$.
On the thinner strip
$\{|\mbox{Re}z|< \frac{3\delta}{2}\}$ the derivative of $\tilde{\mathfrak{g}}_w$
satisfies the inequality $|\tilde{\mathfrak{g}}'_w|\leq \frac{C}{\delta}$, since  $|\tilde{\mathfrak{g}}_w|\leq C$.

Let $F_{\alpha}=[-\frac{1}{2},\frac{1}{2})\times [-\frac{\alpha}{2},\frac{\alpha}{2})\subset \mathbb{C}$ be a fundamental domain for the projection ${\sf P}:\mathbb{C}\to T^{\alpha}$. Put $\Delta^{\alpha,\delta}=F_{\alpha}\cap {\sf P}^{-1}(T^{\alpha,\delta})$.
Let $\chi_0:[0,1]\to \mathbb{R}$ be a non-decreasing function of class $C^2$ with $\chi_0(0)=0,\, \chi_0(1)=1$, $\chi'_0(0)=\chi'_0(1)=0$ and $|\chi_0'(t)|\leq \frac{3}{2}$ (see also Section \ref{sec:3-braids3}). Define $\chi_{\delta}:[\frac{-3\delta}{2},\frac{+3\delta}{2}]\to [0,1]$ by
\begin{equation}\label{eqfin22}
\chi_{\delta}(t) =
\begin{cases} \chi_0(\frac{1}{\delta}t+ \frac{3}{2})& \; t \in [\frac{-3\delta}{2},\frac{-\delta}{2}] \\
1 & \; t \in [\frac{-\delta}{2},\frac{+\delta}{2}]\\
\chi_0(-\frac{1}{\delta}t+ \frac{3}{2})& \; t \in [\frac{\delta}{2},\frac{3\delta}{2}]\,.
\end{cases}
\end{equation}
Notice that $\chi_{\delta}$ is a $C^2$-function that vanishes at the endpoints of the interval  $[\frac{-3\delta}{2},\frac{+3\delta}{2}]$ together with its first derivative, is non-decreasing on $[\frac{-3\delta}{2},\frac{-\delta}{2}]$, and non-increasing on $[\frac{\delta}{2},\frac{3\delta}{2}]$.
Put
$\mathring{\mathfrak{g}}_w(z)= \chi_{\delta} (\mbox{Re}z)\; \tilde{\mathfrak{g}}_w(z) + (1-\chi_{\delta} (\mbox{Re}z))\; \tilde{\mathfrak{g}}_w(0)$ for $z$ in the intersection of
$\Delta^{\alpha,\delta}$ with $\{|\mbox{Re}z|<\frac{3\delta}{2}\}$,
and $\mathring{\mathfrak{g}}_w(z)=\tilde{\mathfrak{g}}_w(0)$ for $z$ in the rest of $\Delta^{\alpha,\delta}$.

Put $\varphi_w(z)= \frac{\partial}{\partial \bar z} \mathring{\mathfrak{g}}_w(z)$ on $\Delta^{\alpha,\delta}$.
Since $\frac{\partial}{\partial \bar z} \chi_{\delta} (\mbox{Re} z)=0$ for $|\mbox{Re} z|<\frac{\delta}{2}$ and for $|\mbox{Re} z|>\frac{3\delta}{2}$,
the function  $\varphi_w(z)$ vanishes on $\Delta^{\alpha,\delta}\setminus Q_{\delta}$ with $Q_{\delta}\stackrel{def}=([ -\frac{3\delta}{2},  +\frac{3\delta}{2}] \times [-\frac{\delta}{2},\frac{\delta}{2}])$.
On $Q_{\delta} \cap \Delta^{\alpha,\delta}$ the inequality
\begin{equation}\label{eqfin23}
|\varphi_w(z)|\leq\frac{1}{2} |\chi'_{\delta}(\mbox{Re} z)|\, | \tilde{\mathfrak{g}}_w(z)- \tilde{\mathfrak{g}}_w(0)|\leq \frac{3}{4\delta} \cdot \frac{C}{\delta}|z|< \frac{3}{4\delta^2} \cdot  C \cdot 2\delta= \frac{3}{2} \frac{C}{\delta}\,
\end{equation}
holds. Notice that the functions $\mathring{\mathfrak{g}}_w$ and $\varphi_w$ extend to ${\sf P}^{-1}( T^{\alpha,\delta})$ as continuous doubly periodic functions. Hence, we may consider them as functions on $T^{\alpha,\delta}$.

\begin{figure}[h]
\begin{center}
\includegraphics[width=60mm]{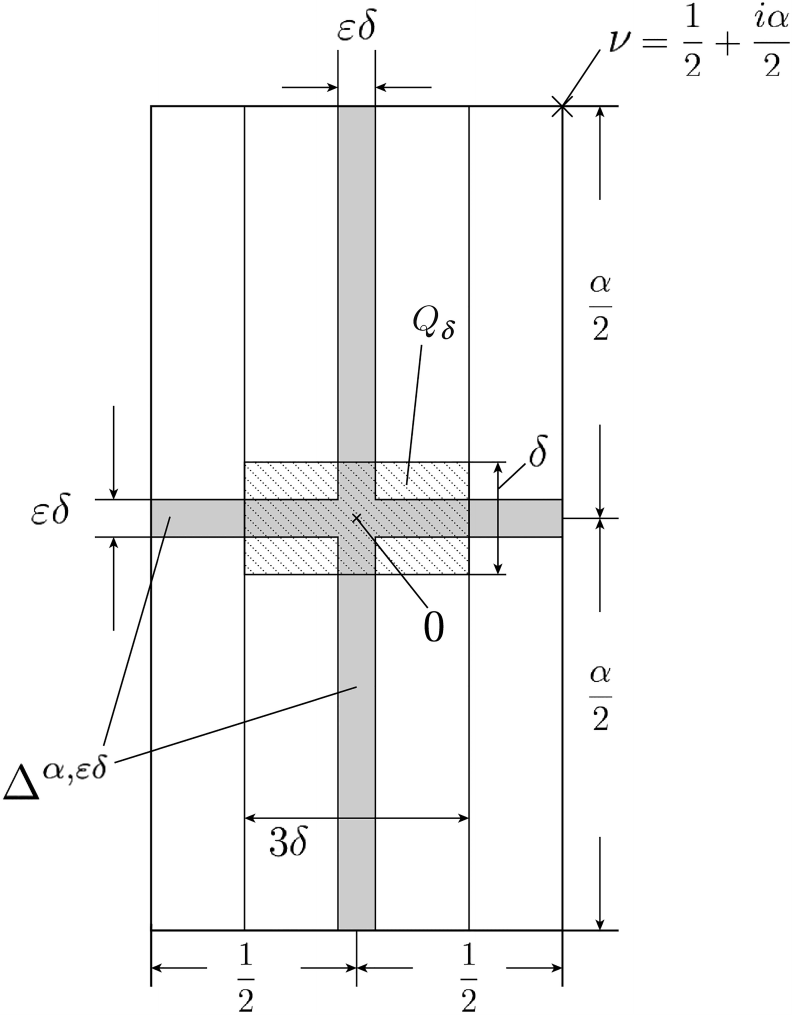}
\end{center}
\caption{A fundamental domain for a torus with a hole and the poles of the kernel for the $\overline{\partial}$-equation}\label{fig8.4}
\end{figure}

We want to find now a small positive number $\varepsilon$ that depends only on $C$,
such that for $\sigma\stackrel{def }= \varepsilon \delta\leq \frac{\varepsilon}{10}$ there exists a solution $f_w$ of the equation $\frac{\partial}{\partial \bar z}f_w(z)=\varphi_w(z)$ on $T^{\alpha,\sigma}$ such that for each $z$ the inequality $|f_w(z)|\leq \frac{1}{C}$ holds.
Then ${\mathfrak{g}}_w-f_w$ is a holomorphic mapping from $T^{\alpha,\sigma}$ to
$\mathbb{C}\setminus \{-1,1\}$ whose class has monodromies equal to \eqref{eqfin21}, and to the identity, respectively.
This provides the $2^{2j}\geq \frac{1}{4}
e^{\frac{\varepsilon\log2}{30\pi}       \frac{\alpha}{\sigma}}$
promised different homotopy classes of mappings from $T^{\alpha,\sigma}$ to $\mathbb{C}\setminus \{-1,1\}$, and, hence proves the lower bound.

To solve the $\bar{\partial}$-problem on $T^{\alpha,\varepsilon\delta}=  T^{\alpha,\sigma}$, we consider an explicit kernel function which mimics the Weierstraß $\wp$-function. The author is grateful to Bo Berndtsson who suggested to use this kernel function.

Recall that the Weierstraß $\wp$-function related to the torus $T^{\alpha}$ is the doubly periodic meromorphic function
$$
\wp_{\alpha}(\zeta)=\frac{1}{\zeta^2}+ \underset{(n,m) \in \mathbb{Z}^2\setminus (0,0)}{\sum} \Big(\frac{1}{(\zeta-n- i m \alpha)^2 }- \frac{1}{(n+ i m \alpha)^2}\Big)\,
$$
on $\mathbb{C}$. It defines a meromorphic function on $T^{\alpha}$ with a double pole at the projection of the origin and no other pole.

Put $\nu=\frac{1}{2} + \frac{\alpha i}{2}$.
Since for $\zeta\not\in (\mathbb{Z}+i\alpha \mathbb{Z})\cup  (\nu+\mathbb{Z}+i\alpha \mathbb{Z})$
the equality
\begin{align*}
\frac{1}{(\zeta-n- i m \alpha) } - \frac{1}{(\zeta-n- i m \alpha-\nu) }+ \frac{\nu}{(\zeta-n- i m \alpha)^2 }= \\
\frac{-\nu^2}{(\zeta -n- i m \alpha)^2 (\zeta - n- i m \alpha-\nu) }\,
\end{align*}
holds, and the series with these terms converges uniformly on compact sets not containing poles, the expression
\begin{align*}
\wp_{\alpha}^{\nu}(\zeta) &\stackrel{def}=
\frac{1}{\zeta}-\frac{1}{\zeta-\nu} \\
& + \,\underset{(n,m) \in \mathbb{Z}^2\setminus (0,0)}{\sum} \Big(\frac{1}{(\zeta-n- i m \alpha) } - \frac{1}{(\zeta-n- i m \alpha-\nu) }     + \frac{\nu}{(n+ i m \alpha)^2}\Big)\,
\end{align*}
defines a doubly periodic meromorphic function on $\mathbb{C}$ with only simple poles. The function descends to a meromorphic function on $T^{\alpha}$ with two simple poles and no other pole.

The support of $\varphi_w$ is contained in $Q_{\delta}$. The set $Q_{\delta}$
is contained in the $2\delta$-disc in $\mathbb{C}$ (in the Euclidean metric) around the origin.
If $\zeta$ is contained in the $2\delta$-disc around the origin and $z \in \Delta^{\alpha,\delta}$, then the point $\zeta-z$ is contained in the ${2\delta}$-neighbourhood (in $\mathbb{C}$) of $\Delta^{\alpha,\delta}$.
By the choice of $\delta$ the distance of any such point  $\zeta-z$
to any lattice point $n+i \alpha m$ except $0$ is larger than  $\frac{1}{2}-2\delta>  \frac{1}{4}$.
Further, for $z\in \Delta^{\alpha,\delta}$ and $\zeta$ in the $2\delta$-disc around
the origin the distance of the point  $\zeta-z$
to any point $n+i \alpha m +\nu$ (including the point $\nu$) is not smaller than $ \frac{1}{2}-\frac{5\delta}{2}\geq  \frac{1}{4}$.
Put $Q'_{\varepsilon,\delta}\stackrel{def}= Q_{\delta}\cap  \Delta^{\alpha,\varepsilon\delta}=      ([  -\frac{3\delta}{2},  +\frac{3\delta}{2}] \times [-\frac{\varepsilon\delta}{2},+\frac{\varepsilon\delta}{2}]) \bigcup ([  -\frac{\varepsilon\delta}{2},  +\frac{\varepsilon\delta}{2}] \times [-\frac{\delta}{2},+\frac{\delta}{2}])$.
Then the function
\begin{align}\label{eqfin24}
f_w(z)= - \frac{1}{\pi}\iint_{Q'_{\varepsilon,\delta}} \varphi_w(\zeta) \wp_{\alpha}^{\nu} (\zeta-z) dm_2(\zeta)\;,
\end{align}
for $z$ in $\Delta^{\alpha,\varepsilon\delta}$ is holomorphic outside  $Q'_{\varepsilon,\delta}$
and satisfies the equation $\frac{\partial}{\partial \bar z}f_w=\varphi_w$ on $Q'_{\varepsilon,\delta}$.
It extends continuously to
a doubly periodic function on ${\sf P}^{-1}(T^{\alpha,\varepsilon\delta})$ and hence descends to a continuous function on $T^{\alpha,\varepsilon\delta}$.
It remains to estimate the supremum norm of the function $f_w$ on  $\Delta^{\alpha,\sigma}=\Delta^{\alpha,\varepsilon\delta}$.
The following inequality holds for  $z \in \Delta^{\alpha,\sigma}$
\begin{align}\label{eqfin26}
|&\iint_{Q'_{\varepsilon,\delta}} \varphi_w(\zeta) \wp_{\alpha}^{\nu} (\zeta-z) dm_2(\zeta)|=\nonumber\\
|& \frac{1}{\pi}\iint_{Q'_{\varepsilon,\delta}} \varphi_w(\zeta)\Big(\frac{1}{\zeta-z} + (\wp_{\alpha}^{\nu}(\zeta-z) -\frac{1}{\zeta-z})\Big)  dm_2(\zeta)|\leq  \nonumber\\
&\frac{1}{\pi}\iint_{Q'_{\varepsilon,\delta}} \frac{3C}{2\delta}\Big(|\frac{1}{\zeta-z}|+ C'\Big)   dm_2(\zeta)  \,.
\end{align}
We used the upper bound \eqref{eqfin23} for $\varphi_w$ and the fact that for $z \in \Delta^{\alpha,\sigma}$ and $\zeta$ in $Q'_{\varepsilon,\delta}$ the expression $|\wp_{\alpha}^{\nu}(\zeta-z)  -\frac{1}{\zeta-z}|$ is bounded by a universal constant $C'$.
The integral  of the second term on the right hand side does not exceed $\frac{3 C}{2 \delta}\cdot C' \cdot 4\varepsilon \delta^2=6 C C' \varepsilon \delta$.
The integral $\quad \iint_{Q'_{\varepsilon,\delta}}|\frac{1}{\zeta-z}|dm_2(\zeta)\; $ does not exceed the sum of the two integrals
$\quad I_1=\iint_{(- \frac{3}{2}  \delta, \frac{3}{2} \delta)\times (- \frac{1}{2}\varepsilon\delta, \frac{1}{2}\varepsilon\delta)}\; \mid \frac{1}{\zeta-z}\mid dm_2(\zeta)\;, \;\; $
and
$\;\; I_2=\iint_{(-\frac{1}{2}\varepsilon\delta, \frac{1}{2}\varepsilon\delta)\times (-\frac{1}{2}\delta,\frac{1}{2}\delta)}\; \mid \frac{1}{\zeta-z}\mid  dm_2(\zeta)\;$.
The first integral $I_1$ is largest when $z= 0 $. Hence, it does not exceed
\begin{align}\label{eqfin27}
\iint_{|\zeta|< (\sqrt{2})^{-1} \varepsilon\delta} \; \mid \frac{1}{\zeta} \mid dm_2(\zeta) + 2\varepsilon \delta \int_{\frac{1}{2}\varepsilon\delta}^{\frac{3}{2}\delta}\, \frac{1}{\eta}\, d\eta
\leq  \sqrt{2}\pi\varepsilon\delta + 2\varepsilon\delta \log{\frac{3}{\varepsilon}}\,.
\end{align}
The second integral $I_2$ is smaller.
We obtain the estimate
\begin{align}\label{eqfin28}
|f_w(z)|\leq & \frac{6 C C' \varepsilon \delta}{\pi} + \frac{3C}{2\pi\delta}(\sqrt{2}\pi\varepsilon\delta + 2\varepsilon\delta \log{\frac{3}{\varepsilon}})\nonumber\\
= & \frac{6 C C'  \delta}{\pi}      \varepsilon + \frac{3C}{2\pi}(\sqrt{2}\pi\varepsilon + 2\varepsilon \log{\frac{3}{\varepsilon}})     \,.
\end{align}
Recall that we have chosen $\delta\leq\frac{1}{10}$.
We may choose $\varepsilon_0>0$ depending only on $C$
so that if $\varepsilon< \varepsilon_0$
the supremum norm of $f_w$ is less than $\frac{1}{C}$.
The proposition is proved.
\hfill  $\Box$

\medskip

\noindent {\bf Proof of Proposition \ref{propfin1b}.} Let $\ell_0$ be the length in the K\"ahler metric of the longest circle in the bouquet. For each natural number $k$ and each positive $\sigma<\sigma_0$ the value $\lambda_k(S_{\sigma})$ satisfies the inequalities
\begin{align}
C_1' \frac{\ell_0}{\sigma}\leq \lambda_k(S_{\sigma})\leq C_1'' \frac{\ell_0}{\sigma}
\end{align}
for constants $C_1'$ and $C_1''$ depending on $k$, $X$, $S$ and the K\"ahler metric. This can be seen by the argument used in the proof of Proposition \ref{propfin1a}.

The upper bound in inequalities \eqref{eqfin4} follows from Theorem \ref{thmfin1}.

The proof of the lower bound in \eqref{eqfin4} will follow along the same lines as the proof of Proposition \ref{propfin1a}. It will lead to a $\overline{\partial}$-problem on an open Riemann surface, for which H\"ormander's $L^2$-method can be used. The case of open Riemann surfaces is easier to treat than the general case of pseudo-convex domains. The needed results for Riemann surfaces are explicitly formulated in \cite{Nap}.

To obtain the lower bound we
consider the circle of the bouquet $S$ with largest length (in the Kähler metric). After a small deformation of the circle we may assume that it is real analytic outside a small neighbourhood of $q$. Take a slightly larger neighbourhood of $q$, so that the part $\gamma_0$ of this circle outside this neighbourhood is connected.
Consider a conformal mapping $\omega$ of a neighbourhood of $\gamma_0$ onto a relatively compact domain $V$ in $\mathbb{C}$, such that $\omega$ takes  $\gamma_0$ to an interval on the imaginary axis.  We may assume that its length is bigger than $3c \ell_0$ for a positive constant $c$, and the mapping is normalized so that $\omega(\gamma_0)$ contains the interval $(-\frac{3c\ell_0}{2},\frac{3c\ell_0}{2})$. Moreover,  we may suppose that the orientation of $\omega(\gamma_0)$
coincides with the orientation of the imaginary axis.

For any positive $\delta$ we denote by $R_{\delta}$ the rectangle
$R_{\delta}\stackrel{def}=\{x+iy: x\in (-\delta,\delta), y \in (-c\ell_0-2\delta,c\ell_0+2\delta)\} $ of extremal length $\lambda(R_{\delta})=c \frac{\ell_0}{\delta} +2$
in the complex plane. Let $\delta_0$ be a positive number such that $S_{2\delta_0}$ is relatively compact in $X$ and is a deformation retract of $X$, and
for all positive $\delta\leq 2\delta_0$ the rectangle  $R_{\delta}$ is contained in $V$.
We denote by $R^X_{\delta}$ the curvilinear rectangle $\omega^{-1}(R_{\delta})$ in $X$.

Let $\delta\leq\delta_0$. Put
$\mathring{R}_{\delta}\stackrel{def}= \{x+iy: x\in (-\delta,\delta), y \in (-c\ell_0,c\ell_0)\}
\subset R_{\delta} $ be the rectangle in the complex plane with the same center and horizontal side length as $R_{\delta} $, with
vertical side length $2c\ell_0$ and extremal length $\lambda(\mathring{R}_{\delta})= \frac{c\ell_0}{\delta}$.
Put $j\stackrel{def}= [\frac{c\ell_0}{12\delta \pi}]$.
By
Remark \ref{rem3-braids.1}
we may represent
any word $w$ of the form $a_1^{\pm 2}\, a_2^{\pm 2}\ldots a_2^{\pm 2}$ in the relative fundamental group  $\pi_1(\mathbb{C}\setminus\{-1,1\}, (-1,1))$ with $2j$  terms
by a holomorphic mapping $g_w:\mathring{R}_{\delta}\to  \{z\in\mathbb{C}: |z|<C, |z\pm 1|>\frac{1}{C}\}         \subset \mathbb{C}\setminus \{-1,1\}$
that vanishes at $ \pm i c\ell_0$.
The mapping $g_w$ extends by reflection through the horizontal sides of $\mathring{R}_{\delta}$ to a holomorphic function on $R_{\delta}$, that we also denote by $g_w$. Since $|g_w|\leq C$ on $\mathring{R}_{\delta}$ for the constant $C$ from
Remark \ref{rem3-braids.1},
hence, $|g_w|\leq C$ also on $R_{\delta}$, for any positive $\alpha<1$ the inequality $|g_w'|\leq  \frac{C}{\delta(1-\alpha)}$ holds for the derivative of the mapping $g_w$ on the smaller rectangle $R_{\alpha\delta}$ (defined as $R_\delta$ with $\delta$ replaced by $\alpha\delta$). This fact implies that $|g_w|\leq \frac{\sqrt{2} C\alpha}{1-\alpha}$ on $Q_{\alpha\delta}^{\pm}\stackrel{def}=
\{x+iy: x\in (-{\alpha\delta},{\alpha\delta}), \pm y\in
(c\ell_0, c\ell_0 + \alpha \delta)\}$. We took into account that
$g_w(\pm i c \ell_0)=0$.
We take $\alpha$ so that  $\frac{\sqrt{2} C\alpha}{1-\alpha} <1-\frac{1}{C}$.

With the same function $\chi_0$ as in the proof of Proposition \ref{propfin1a} we define
\begin{equation}\label{eqfin22+}
\mathring{\chi}_{\delta}(t) =
\begin{cases}
1 & \; \;t \;\in [-c\ell_0,c\ell_0]\\
\chi_0( \frac{c\ell_0+ \alpha\delta -|t|}{\alpha\delta}) &\; |t| \in (c\ell_0, c\ell_0+ \alpha \delta)\,,\\
0&\; \;t\; \in \mathbb{R}\setminus[-c\ell_0-\alpha\delta,c\ell_0+ \alpha\delta]\, .
\end{cases}
\end{equation}

Consider the function $\tilde{g}_w(z)\stackrel{def}=g_w(z)\cdot\mathring{\chi}_{\delta}({\rm Im}(z))$ and the
continuous $(0,1)$-form $\varphi_w\stackrel{def}= \bar{\partial}{\tilde{g}_w}$ on  $R_{\alpha\delta}$. The form $\varphi_w$ vanishes outside $Q_{\alpha\delta}^{\pm}$.
Let $\varepsilon$
be a small positive number that does  not depend on $\delta$ and will be chosen later.
Consider the measurable $(0,1)$-form $\varphi_{w,\varepsilon}$ on $R_{\delta}$ that equals $\varphi_w$ on $Q_{\alpha\delta,\varepsilon}^{\pm}\stackrel{def}=
\{x+iy: x\in (-{\alpha\delta\varepsilon},{\alpha\delta\varepsilon}), \pm y\in
(c\ell_0, c\ell_0 + \alpha \delta)\}$
and vanishes outside this set.
Extend its pullback under the conformal mapping $\omega: R^X_{\delta}\to R_{\delta}$
to a measurable $(0,1)$-form on $X$ by putting it equal to zero outside $R^X_{\delta}$. Denote the obtained form by $\varphi_{w,\varepsilon}^X$.

By Corollary 2.14.2
of \cite{Nap} there exists a strictly subharmonic exhaustion function $\psi$ on $X$.
The $L^2$-norm of $\varphi_{w,\varepsilon}$ with respect to the Euclidean metric on the complex plane does not exceed $C_2 \sqrt{\varepsilon}$ for an absolute constant $C_2$.
Hence, the weighted $L^2$-norm on $X$ of $\varphi^X_{w,\varepsilon}$ with respect to the K\"ahler metric and the weight $e^{-\psi}$ (see Definition 2.6.1
of \cite{Nap}) does not exceed $C_3 C_2 \sqrt{\varepsilon}$ for a constant $C_3$ that depends on $\psi$ and
on the K\"ahler metric on the relatively compact subset
$V  \supset R^X_{\delta}  $ of $X$. By Corollary 2.12.6 of \cite{Nap}
there exists a function $f_{w,\varepsilon}^X$ with $\bar{\partial} f_{w,\varepsilon}^X=\varphi^X_{w,\varepsilon}$ in the weighted $L^2$-space on $X$ with respect to the K\"ahler metric and the weight $e^{-\psi}$ (see Definition 2.6.1
of \cite{Nap}), whose norm in this space does not exceed $C_4 C_3 C_2 \sqrt{\varepsilon}$ for a constant $C_4$ depending only on $X$, $\psi$, and the K\"ahler metric.

Let $(Q_{\alpha\delta,\varepsilon}^{\pm})^X$ be the preimages of $Q_{\alpha\delta,\varepsilon}^{\pm}$ under $\omega$.
The function $f_{w,\varepsilon}^X$ is holomorphic on
$X\setminus \big(\overline{(Q_{\alpha\delta,\varepsilon}^{+})^X}\cup \overline{(Q_{\alpha\delta,\varepsilon}^{-})^X}\big)$.
For the above chosen constant $\delta_0$ we put $\tilde{Q}_{\delta_0}^{\pm}\stackrel{def}=\{x+iy\in R_{\delta_0}: \pm y\in (c\ell_0 -\delta_0, c\ell_0 + 2\delta_0)\}$, and $(\tilde{Q}_{\delta_0}^{\pm})^X=\omega^{-1}(\tilde{Q}_{\delta_0}^{\pm})$. Then $(Q_{\alpha\delta,\varepsilon}^{\pm})^X$ is relatively compact in $(\tilde{Q}_{\delta_0}^{\pm})^X$.
On a relatively compact open subset $X_0$ of $X$, that
contains the closed subset $\overline{S_{2\delta_0}} \setminus  ((\tilde{Q}_{\delta_0}^+)^X \cup  (\tilde{Q}_{\delta_0}^-)^X)$ of $X$, the supremum norm of $|f_{w,\varepsilon}^X|$ is estimated by its weighted $L^2$-norm: $|f_{w,\varepsilon}^X|< C_5 \sqrt{\varepsilon}$ in a neighbourhood of  $\overline{S_{2\delta_0}} \setminus  ((\tilde{Q}_{\delta_0}^+)^X \cup  (\tilde{Q}_{\delta_0}^-)^X)$ for a constant $C_5$ that depends on the relatively compact open subset $V$ of $X$, on the K\"ahler metric, on $\psi$ and on the constants chosen before (see Theorem 2.6.4
of \cite{Nap}).

On the other hand the classical Cauchy-Green formula on the complex plane provides a solution $\tilde{f}_{w,\varepsilon}$ of the equation $\overline{\partial}\tilde{f}_{w,\varepsilon}=\varphi_{w,\varepsilon}$ on the set $\tilde{Q}_{\delta_0}^+ \cup  \tilde{Q}_{\delta_0}^-$. The supremum norm of the function $\tilde{f}_{w,\varepsilon}$ is estimated by $C_6\sqrt{\varepsilon}$ for an absolute  constant $C_6$.
Let $\tilde{f}_{w,\varepsilon}^X$ be the pullback of $\tilde{ f}_{w,\varepsilon}$ to  $(\tilde{Q}_{\delta_0}^+)^X \cup  (\tilde{Q}_{\delta_0}^-)^X$. The function $f_{w,\varepsilon}^X - \tilde{f}_w^X$ is holomorphic on $(\tilde{Q}_{\delta_0}^+)^X \cup  (\tilde{Q}_{\delta_0}^-)^X$ and satisfies the inequality $|f_{w,\varepsilon}^X - \tilde{f}_{w,\varepsilon}^X|<(C_5+C_6)\sqrt{\varepsilon}$ at all points of the set $(\tilde{Q}_{\delta_0}^+)^X \cup  (\tilde{Q}_{\delta_0}^-)^X$, that are close to its boundary. Hence, the inequality is satisfied on $(\tilde{Q}_{\delta_0}^+)^X \cup  (\tilde{Q}_{\delta_0}^-)^X$. As a consequence,
\begin{align}\label{eqfin19'}
|f_{w,\varepsilon}^X|<(C_5+2C_6)\sqrt{\varepsilon}\,\; \mbox{on}\,\; (\tilde{Q}_{\delta_0}^+)^X \cup  (\tilde{Q}_{\delta_0}^-)^X\,.
\end{align}
Choose now $\varepsilon$ depending on $C_5$ and $C_6$, so that
\begin{align}\label{eqfin19}
(C_5+2C_6)\sqrt{\varepsilon}<\frac{1}{C}\,.
\end{align}
Then
\begin{align}\label{eqfin19''}
|f_{w,\varepsilon}^X|< \frac{1}{C} \; \mbox{on}\;  \overline{S_{2\delta_0}}\,.
\end{align}
Put $\sigma_0=\varepsilon\alpha \delta_0$, and take
$\sigma=\varepsilon\alpha \delta\leq\sigma_0$. Consider the smooth function $g_{w,\sigma}^X$ on  $S_{\alpha\delta\varepsilon }=S_{\sigma}$ which equals
the pullback $\tilde{g}_w^X= \tilde{g}_w\circ\omega$ on $ \omega^{-1} (R_{\alpha \delta})\cap S_{\sigma}$,
and vanishes on the rest of $S_{\sigma}$. The function $g_{w,\sigma}^X$  vanishes on all circles of the bouquet except $\gamma_0$, and therefore, the monodromy of its homotopy class along each circle of the bouquet except $\gamma_0$ is the identity. The restriction of $g_{w,\sigma}^X$ to
$\mathring{R}_{\delta}^X\cap S_{\sigma}=\omega^{-1}(\mathring{R}_{\delta})\cap S_{\sigma}
$ represents the element $a_1^{\pm 2} a_2^{\pm 2}\ldots a_2^{\pm 2}\in\pi_1(\mathbb{C}\setminus\{-1,1\}, (-1,1))$. Moreover, on $\big({R}_{\delta}^X\setminus \mathring{R}_{\delta}^X\big)\cap S_{\sigma} $  the inequality $|g_{w,\sigma}^X|<1-\frac{1}{C} $ holds. Hence, the monodromy of the homotopy class of $g_{w,\sigma}^X$ along $\gamma_0$ equals $a_1^{\pm} a_2^{\pm}\ldots a_2^{\pm}$.
By inequality
\eqref{eqfin19''} the  restriction $(g_{w,\sigma}^X   -f_{w,\varepsilon}^X)\mid S_{\sigma}$ is a holomorphic mapping into $\mathbb{C}\setminus \{-1,1\}$ and
its monodromies
along all circles of the bouquet coincide with those of the homotopy class of $g_{w,\sigma}^X$.
For each positive $\sigma=\epsilon\alpha\delta <\epsilon\alpha\delta_0=\sigma_0$ we found no less than
$2^{2j}\geq\frac{1}{4}e^{\frac{    c\alpha\varepsilon \log 2}{ 6 \pi} \frac{\ell_0}{\sigma}}$
irreducible non-homotopic holomorphic mappings from $S_{\sigma}$ to $\mathbb{C}\setminus\{-1,1\}$.
The proposition is proved. \hfill $\Box$

\medskip

\noindent {\bf Proof of Proposition  \ref{propfin5}.} The first part of the proposition  follows along the lines of the proof of the lower bound in Proposition  \ref{propfin1b}.

Here is a sketch of the proof of the second part of Proposition  \ref{propfin5}. Let $f:S_{\sigma}\to \mathbb{C}\setminus \{-1,1\}$ be a holomorphic mapping with $f(q_0)=0$. There exist positive constants $c$ and $c'<\frac{1}{2}$ depending only on the Kähler metric,
such that $f$ maps the $c\sigma$-neighbourhood $D_{c\sigma}$ of $q_0$ in $S_{\sigma}$ (in the Kähler metric) into the disc of radius $c'$ with center $0$
in $ \mathbb{C}\setminus \{-1,1\}$ (in the Euclidean metric). Indeed, For $c=1$ we take a conformal mapping $\omega$ from the unit disc $\mathbb{D}$ in the complex plane onto the $\sigma$-neighbourhood  $D_{\sigma}$ of $q_0$ with $\omega(0)=q_0$. The holomorphic mapping $f\circ\omega$ takes $\mathbb{D}$ into $ \mathbb{C}\setminus \{-1,1\}$ and $0$ to $0$. It lifts to a
holomorphic mapping $\widetilde{f\circ{\omega}}$ from $\mathbb{D}$ to the right half-plane $\mathbb{C}_r$ that takes $0$ to $\frac{1+i}{2}$. Let $\mathfrak{c}$ be a  conformal mapping from the right half-plane $\mathbb{C}_r$ onto the unit disc $\mathbb{D}$ that takes  $\frac{1+i}{2}$ to $0$. The composition $\mathfrak{c}\circ \widetilde{f\circ{\omega}}$ maps the unit disc to the unit disc and takes $0$ to $0$. Hence,
for the derivative of this mapping the inequality $|(\mathfrak{c}\circ \widetilde{f\circ{\omega}})'|\leq 2$ holds on the disc  $\frac{1}{2}\mathbb{D}$  of radius $\frac{1}{2}$ around $0$. Using the properties of $\mathfrak{c}$ and of the projection $\mathbb{C}_r\to  \mathbb{C}\setminus \{-1,1\}$ we find the constant
$c'$.

We have the isomorphisms of fundamental groups
$\pi_1(X,q_0)\cong \pi_1(X,D_{c\sigma})$ and $\pi_1(\mathbb{C}\setminus \{-1,1\},0) \cong \pi_1(\mathbb{C}\setminus \{-1,1\},\mathbb{D}_{c'})$.
For each $j$ we take a curvilinear rectangle $R_j$ contained in $S_{c\sigma}$ that has horizontal sides in $D_{c\sigma}$ and extremal length not exceeding $C_3\frac{\ell_j}{\sigma}$ for a constant $C_3$ depending on the Kähler metric,
and represents $e_j$ with horizontal boundary values in $D_{c\sigma}$. The restriction $f\mid R_j$ represents $f_*(e_j)$ with horizontal boundary values in $\mathbb{D}_{c'}$.

If $c'$ is small then for the mappings $f_1$ and $f_2$ of Proposition \ref{prop2} the preimage $(f_1\circ f_2)^{-1}(\mathbb{D}_{c'})$ is equal to the union $\bigcup_{k\in\mathbb{Z}}\big(\mathbb{D}_{c''}+ik\big)$ for a positive constant $c''<\frac{1}{2}$. As a consequence,
for each elementary word $w$ (including the exceptional cases)  the extremal length $\Lambda(_{\mathbb{D}_{c'}}w_{\;\mathbb{D}_{c'}})$
with boundary values in $\mathbb{D}_{c'}$ is not smaller than ${\rm const}\, \mathcal{L}(w)$ for an absolute constant ${\rm const}$. Moreover, the extremal lengths with mixed horizontal boundary values $\Lambda(_{\mathbb{D}_{c'}}w_{\,tr})$, $\Lambda(_{\mathbb{D}_{c'}}w_{\,pb})$, $\Lambda(_{tr}w_{\;\mathbb{D}_{c'}})$, $\Lambda(_{pb}w_{\;\mathbb{D}_{c'}})$ are not smaller than ${\rm const}\, \mathcal{L}(w)$. These facts can be proved similarly as Proposition \ref{prop4b}.

The estimates $\Lambda(_{\mathbb{D}_{c'}}(f_*(e_j)_{\,\mathbb{D}_{c'}})\geq C_4 \mathcal{L}(f_*(e_j)),\; j=1,\ldots,2g+m,$ for a constant $C_4$ depending on $c''$ can now be proved following along the lines of the proof of Theorems \ref{thm1} and \ref{thm10.1'}. By the choice of the curvilinear rectangles $R_j$   inequalities \eqref{eqfin31} hold. \hfill $\Box$

\appendix

\chapter{Several complex variables}
\label{ChapterA}

\medskip

\begin{thm}\label{thmA1}({\bf \cite{H1}, Theorem 5.5.1})
Let $\Omega$ be a Stein manifold and $\Omega_j$ open subsets of $\Omega$ such that
$\Omega=\bigcup_{j=1}^{\infty}\Omega_j$. If $g_{j,k}, j,k=1,2,\ldots,$ are holomorphic functions on $ \Omega_j \cap \Omega_k$ such that
\begin{align}\label{eqA1}
g_{j,k}= &-g_{k,j} \;\mbox{for all}\; j,k, \; \nonumber\\
g_{j,k}+g_{k,i}+g_{i,j}=&0 \, \mbox{on} \;  \Omega_j \cap \Omega_k \cap \Omega_i\; \mbox{for all}\; j,k,i,
\end{align}
then there are holomorphic functions $g_j$ on $\Omega_j$ such that
\begin{equation}\label{eqA2}
g_{j,k}=g_j-g_k  \;\mbox{for all}\; j,k.
\end{equation}
\end{thm}
A family of holomorphic functions $g_{j,k}$ on $ \Omega_j \cap \Omega_k$ that satisfies \eqref{eqA1} is called a Cousin I distribution. If
there are holomorphic functions $g_j$ on $\Omega_j$ satisfying \eqref{eqA2} we say that the first Cousin Problem for the Cousin distribution has a solution.

A family of nowhere vanishing holomorphic functions $g_{j,k}$ on $ \Omega_j \cap \Omega_k$ is called a Cousin II distribution if
\begin{align}\label{eqA3}
g_{j,k}g_{k,j} = &1 \;\mbox{for all}\; j,k, \; \nonumber\\
g_{j,k}g_{k,i}g_{i,j} = &1 \, \mbox{on} \;  \Omega_j \cap \Omega_k \cap \Omega_i\; \mbox{for all}\; j,k,i.
\end{align}
If there are nowhere vanishing holomorphic functions $g_j$ on $\Omega_j$ satisfying
the equation
\begin{equation}\label{eqA4}
g_{j,k}=g_jg_k^{-1}  \;\mbox{for all}\; j,k,
\end{equation}
we say that the second Cousin Problem for the Cousin II problem has a solution.

If the $\Omega_j$ cover a Stein manifold and all sets $\Omega_j \cap \Omega_k$ are simply connected, then the respective second Cousin Problem has a solution. Indeed, for each $j,k$
we may consider a branch of the logarithm of $g_{j,k}$ on $\Omega_j \cap \Omega_k$ and solve a Cousin I Problem for these functions.

\bigskip

\chapter{A Lemma on Conjugation}
\label{ChapterB}

\medskip

\begin{lemm}\label{lemm6.4}{\bf{(Lemma on conjugation)}}

\noindent $1$. Let $Q_1 , \ldots , Q_k$ be topological spaces, and let $\varphi$ be a
self-homeomorphism of (the formal disjoint union)
$Q\stackrel{def}=\bigcup_{1\leq j\leq k} Q_j$.  Suppose $\varphi$  permutes the $Q_j$
cyclically, i.e. $\varphi (Q_j) = Q_{j+1}$, $j = 1,\ldots , k-1$,
$\varphi (Q_k) = Q_1$.

Let $\mathring{Q}_1 , \ldots ,\mathring{ Q}_k$ be topological spaces, and
let $\mathring{\varphi}$ be a
self-homeomorphism of (the formal disjoint union)
$\mathring{Q}\stackrel{def}=\bigcup_{1\leq j\leq k} \mathring{Q}_j$, that permutes the $\mathring{Q}_j$ along the respective cycle,
i.e. $\mathring{\varphi} (\mathring{Q}_j) = \mathring{Q}_{j+1}$, $j = 1,\ldots , k-1$,
$\mathring{\varphi} (\mathring{Q}_k) = \mathring{Q}_1$.
Assume that $\varphi^k \mid Q_1$
and $\mathring{\varphi}^k \mid \mathring{Q}_1$ are related by isotopy and conjugation.

Then $\varphi$ and $\mathring{\varphi}   $ are related by isotopy followed by
conjugation, i.e. there exists a homeomorphism
$\sigma$ from $\mathring{Q}$ onto $Q$,
such that $\sigma_j\stackrel{def}=\sigma|\mathring{Q}_j$ maps each $\mathring{Q}_j$ homeomorphically onto ${Q}_j$, $j=0,\ldots,k$, and a self-homeomorphism $\tilde{ \varphi}$ of $Q$,
which is isotopic to $\varphi$ and such that $ \mathring{\varphi}= \sigma^{-1} \circ \tilde{\varphi} \circ \sigma$.

\smallskip
\noindent $2$. Suppose in addition, that $\mathring{Q}_j=Q_j$ for each $j=1,\ldots,k$, and there are
 subsets $Q'_j$ of $Q_j$, $j=1,\ldots,k,$ such that $\varphi= \mathring{\varphi}$ on
$Q'\stackrel{def}=\bigcup_{1\leq j\leq k}Q'_j$ and $\varphi\mid Q'$ is a self-homeomorphism of $Q'$.
Moreover, assume that $\varphi^k \mid Q_1$ is isotopic to $\mathring{\varphi}^k \mid \mathring{Q}_1$ by an isotopy of the form $\varphi^k\circ \tilde{\chi}_t,\,t\in[0,1],$ for a ${\rm Hom}^+(Q_1;Q_1')$-isotopy   $\tilde{\chi}_t$.

Then the mapping $\sigma$ of Statement 1 may
be chosen to fix the union  $Q'=\bigcup_{1\leq j\leq k}Q'_j$ pointwise and the self-homeomorphism $\tilde{\varphi}$ of Statement 1 may be chosen to be isotopic to $\varphi$ by an isotopy of the form $\varphi \circ \chi_t,\, t \in[0,1],$ for a ${\rm Hom}^+(Q;Q')$-isotopy $\chi_t$
with $\chi_0=\rm{ Id}$.
\end{lemm}

\bigskip

\noindent {\bf Proof of the Lemma on conjugation.}
Assume first that in Statement 1
\begin{equation}\label{eq6.66b}
\mathring{\varphi}^k\mid \mathring{Q}_1=\sigma_1^{-1}\circ \varphi^k\circ \sigma_1\mid \mathring{Q}_1
\end{equation}
for a homeomorphism $\sigma_1$ from $\mathring{Q}_1$ onto $Q_1$.
Put $\varphi_j =
\varphi \mid Q_j$ and $\mathring{\varphi}_j=\mathring{\varphi}\mid \mathring{Q}_j$ for $j=1,\ldots,k$.
The conditions of the lemma give the following commutative diagram:
$$
\xymatrix{
\mathring{Q}_1 \ar[rr]^{\mathring{\varphi}_1} \ar[d]^{\sigma_1} &&\mathring{Q}_2
\ar[rr]^{\mathring{\varphi}_2} &&\ldots \ar[rr]^{\mathring{\varphi}_{k-1}} &&\mathring{Q}_k
\ar[rr]^{\mathring{\varphi}_k} &&\mathring{Q}_1 \ar[d]^{\sigma_1} \\
Q_1 \ar[rr]^{\varphi_1} &&Q_2 \ar[rr]^{\varphi_2} &&\ldots
\ar[rr]^{\varphi_{k-1}} &&Q_k \ar[rr]^{\varphi_k} &&Q_1 }
$$

Let $\sigma_2:\mathring{Q}_2\to Q_2$ be the homeomorphism for which the following diagram commutes:
$$
\xymatrix{
\mathring{Q}_1 \ar[rr]^{\mathring{\varphi}_1} \ar[d]^{\sigma_1} &&\mathring{Q}_2\ar[d]^{\sigma_2} \\
Q_1 \ar[rr]^{\varphi_1} &&Q_2}
$$

In other words, we put $\sigma_2= \varphi_1\circ \sigma_1 \circ\mathring{\varphi}_1^{-1}$.
If $\sigma_j$ is defined for all $j\leq j_0\leq k-1$, we define $\sigma_{ j_0+1}= \varphi_{j_0}\circ \sigma_{j_0}\circ \mathring{\varphi}_{j_0}^{-1}$. Then for all $j\leq k$ the equality
\begin{align}\label{eqB.2}
\sigma_j= (\prod_{j'=1}^{j-1}  \varphi_{j'}) \circ  \sigma_1 \circ (\prod_{j'=1}^{j-1} \mathring{\varphi}_{j'})^{-1}
\end{align}
holds. Since $\prod_{j'=1}^k \varphi_{j'}= \varphi^k| Q_1$ and $\prod_{j'=1}^k \mathring{\varphi}_{j'}=\mathring{\varphi}^k|\mathring{Q}_1$
equation \eqref{eq6.66b} implies that the following diagram commutes:

$$
\xymatrix{
\mathring{Q}_1 \ar[rr]^{\mathring{\varphi}_1} \ar[d]^{\sigma_1} &&\mathring{Q}_2 \ar[d]^{\sigma_2} \ar[rr]^{\mathring{\varphi}_2} &&\ldots \ar[rr]^{\mathring{\varphi}_{k-1}} &&\mathring{Q}_k \ar[d]^{\sigma_k} \ar[rr]^{\mathring{\varphi}_k} &&Q_1 \ar[d]^{\sigma_1} \\
Q_1 \ar[rr]^{\varphi_1} &&Q_2 \ar[rr]^{\varphi_2} &&\ldots
\ar[rr]^{\varphi_{k-1}} &&Q_k \ar[rr]^{\varphi_k} &&Q_1 }
$$

Let $\sigma:\mathring{Q}\to Q$ be the mapping that equals $\sigma_j$ on $\mathring{Q}_j$.
The commutativity of the diagram means that $\mathring{\varphi}=\sigma^{-1}\circ \varphi \circ \sigma$. Indeed, the equality
\begin{equation}
\label{eq6.68bis} \mathring{\varphi} = \sigma^{-1} \circ \varphi \circ \sigma
\quad {\rm on} \quad \bigcup_{j=1}^k \ Q_j
\end{equation}
is equivalent to the system of equations
\begin{eqnarray}
\label{eq6.69}
\mathring{\varphi}_1 &= &\sigma_2^{-1} \circ \varphi_1 \circ \sigma_1 \nonumber \\
\vdots && \nonumber \\
\mathring{\varphi}_{k-1} &= &\sigma_k^{-1} \circ \varphi_{k-1} \circ \sigma_{k-1}  \\
\mathring{\varphi}_k &= &\sigma_1^{-1} \circ \varphi_k \circ \sigma_k \, .
\nonumber
\end{eqnarray}

The Statement 1 of the Lemma 6.1 is clear for the case when equality \eqref{eq6.66b} holds.

Consider now the general case of Statement 1, i.e. we assume that
$$
\mathring{\varphi}^k\mid \mathring{Q}_1 = \sigma_1^{-1}\circ \widetilde{\varphi^k}\circ \sigma_1
$$
on $\mathring{Q}_1$ for a homeomorphism $\sigma_1$ from $\mathring{Q}_1$ onto $Q_1$, and
a self-homeomorphism $\widetilde{\varphi^k}$ of $Q_1$ that is
isotopic to the self-homeomorphism  $\varphi^k\mid Q_1$   of $Q_1$.
To prove the first statement of the lemma, it is enough to
find a self-homeomorphism $\tilde\varphi$ of $Q$ that is obtained from $\varphi$ by isotopy and conjugation, for which $\tilde\varphi^k= \widetilde{\varphi^k}$ on $Q_1$. Then by the arguments for the case when equality \eqref{eq6.66b} holds $\mathring{\varphi}$ is conjugate to $\tilde\varphi$ and the first statement is proved in the general situation.

We put $\tilde\varphi=\varphi$ on $\bigcup_{1<j\leq k}Q_j$ and $\tilde\varphi= (\tilde\varphi)^{-k+1}\circ \widetilde{\varphi^k}=(\varphi)^{-k+1}\circ \widetilde{\varphi^k}$ on $Q_1$.
Then $(\tilde\varphi)^k=\widetilde{\varphi^k}$ on $Q_1$.
The isotopy joining $\varphi$ and $\tilde\varphi$ is described as follows.
It is equal to $ \varphi\circ\chi_t$ on $\bigcup_{1\leq j\leq k}Q_j$ with $\chi_t$ being the identity mapping on $\bigcup_{1< j \leq k}Q_j$ for $t\in[0,1]$, and on $Q_1$ the
$\chi_t$ form an isotopy of self-homeomorphisms of $Q_1$ for which $\chi_0$ is the identity and $\chi_1$ is equal to the mapping $\chi_1= (\varphi)^{-k}\circ (\widetilde{\varphi^k})$. Then $\varphi\circ\chi_t\mid Q_k$ is equal to $\varphi$ for $t=0$ and to $(\varphi)^{-k+1}\circ\widetilde{\varphi^k} = \tilde \varphi$ for $t=1$.

Statement 2 is obtained as follows. We look again for a suitable self-homeomorphism $\tilde {\varphi}$ of $Q$ that is isotopic to $\varphi$ and satisfies the equality $\tilde {\varphi}^k\mid Q_1=\mathring{\varphi}^k\mid Q_1 $.
As in the proof of Statement 1 we put  $\tilde {\varphi}=\varphi$ on $\bigcup_{1< j\leq k} Q_j$ and $\tilde {\varphi}= (\varphi^{-k+1})\circ  \mathring{\varphi} ^k $ on $Q_1$. Then $\mathring{\varphi}^k\mid Q_1 = \tilde {\varphi}^k \mid Q_1$ and $\mathring{\varphi}=\tilde{\varphi}$ on $\bigcup_{1\leq j<k} Q'_j$.

We prove now that $\varphi$ and  $\tilde {\varphi}$ are related by the required isotopy.
Put $\chi_t={\rm Id}$ on $\bigcup_{1< j\leq k}Q_j$ for all $t\in [0,1]$, and $\chi_t=\tilde{\chi}_t$ on $Q_1$. Here $\tilde{\chi}_t$ is the ${\rm Hom}^+(Q_1,Q_1')$-isotopy in the assumption of Statement 2 of the lemma, $\tilde{\chi}_0=\rm{Id}$ and $\tilde{\chi}_1= \varphi^{-k} \circ\mathring{\varphi}^k$. Then $\chi_t$ is a ${\rm Hom}^+(Q,Q')$-isotopy.
The family $ \varphi \circ \tilde{\chi}_t $
joins $\varphi=\varphi \circ \chi_0$ with  $\tilde{\varphi}=\varphi \circ \chi_1$. Indeed, this is true on  $\bigcup_{1< j\leq k} Q_j$, since there $\tilde{\varphi}=\varphi$ and $\chi_t={\rm id}$. On $Q_1$ $\varphi \circ \chi_0=\varphi$ and $\varphi \circ \chi_1=\varphi\circ\varphi^{-k} \circ\mathring{\varphi}^k=\tilde {\varphi}$. Finally $\tilde{\varphi}^k\mid Q_1= \mathring{\varphi}^k\mid Q_1$.

The homeomorphism $\sigma$ that conjugates $\tilde{\varphi}$ to $\mathring{\varphi}$ is given by equality \eqref{eqB.2} (see the proof of Statement 1). Since under the present conditions $\sigma_1$ is the identity and the $\tilde{\varphi}_j$ and $\mathring{\varphi}_j$ coincide on $Q_j'$, $\sigma$ equals the identy on $Q'$.
The lemma is proved. \hfill $\Box$

\chapter{Koebe's Theorem}
\label{ChapterC}
\begin{thm}\label{thmKoebe}{\bf (Koebe's $\frac{1}{4}$-Covering Theorem.)} {\rm (\cite{Go}, Theorem 2, II,§4)}
Suppose a holomorphic function $f$ on the unit disc $\mathbb{D}$ in the complex plane with $f(0)=0$ and $f'(0)=1$ is univalent (in other words, it maps the unit disc conformally onto its image). Then the image $f(\mathbb{D})$ contains the disc of radius $\frac{1}{4}$ with center $0$.
\end{thm} 
\printindex

\backmatter
\bibliographystyle{amsalpha}



\end{document}